\documentclass[12pt, a4paper, openany]{memoir}

\usepackage{fouriernc}
\usepackage[T1]{fontenc}

\usepackage{etoolbox}

\settrimmedsize{297mm}{210mm}{*}
\settypeblocksize{634pt}{448.13pt}{*} 
\setulmargins{4cm}{*}{*} 
\setlrmargins{*}{*}{1.5}
\setmarginnotes{17pt}{51pt}{\onelineskip}
\setheadfoot{\onelineskip}{2\onelineskip} 
\setheaderspaces{*}{2\onelineskip}{*} 
\checkandfixthelayout
\makepagestyle{Template}
\makeoddfoot{Template}{}{\thepage}{}
\makeevenfoot{Template}{}{\thepage}{}
\makeheadrule{Template}{\textwidth}{\normalrulethickness}
\makeevenhead{Template}{\small\textsc{\leftmark}}{}{}
\makeoddhead{Template}{}{}{\small\textsc{}}
\renewenvironment{abstract}
 {\small
  \begin{center}
  \bfseries \abstractname\vspace{-.5em}\vspace{0pt}
  \end{center}
  \list{}{%
    \setlength{\leftmargin}{10mm}
    \setlength{\rightmargin}{\leftmargin}%
  }%
  \item\relax}
 {\endlist}

\usepackage{calc}
\newif\ifNoChapNumber
\makeatletter
\makechapterstyle{Template}{

}
\chapterstyle{Template}

\usepackage{titlesec}
\titleformat{\section}{\large\bfseries}{\thesection}{0.5em}{}
\titleformat{\subsection}{\large\bfseries}{\thesubsection}{0.5em}{}
\titleformat{\subsubsection}[runin]{\normalsize\bfseries\filright}{\thesubsubsection}{0.5em}{}[.]
\titleformat{name=\subsubsection, numberless}[runin]{\normalsize\bfseries\filright}{\indent}{0.5em}{}[.]
\setsecnumdepth{section} 
\maxsecnumdepth{subsubsection}
\maxtocdepth{subsection}

\usepackage{graphicx}
\newcommand\MyBox[1]{%
    \fbox{\parbox[c][2.5cm][c]{4cm}{\centering #1}}%
}
\newcommand\MyVBox[1]{%
    \parbox[c][2cm][c]{2cm}{\centering\bfseries #1}%
}  
\newcommand\MyHBox[2][\dimexpr4cm+2\fboxsep\relax]{%
    \parbox[c][1cm][c]{#1}{\centering\bfseries #2}%
}  
\newcommand\MyTBox[4]{%
    \MyVBox{#1}
    \MyBox{#2}\hspace*{-\fboxrule}%
    \MyBox{#3}\par\vspace{-\fboxrule}%
}  

\usepackage{amsmath}
\usepackage{amsthm}
\usepackage{amssymb}
\usepackage{bm}
\usepackage{mathtools}
\usepackage{rotating}
\usepackage{scalerel}
\usepackage{tikz-cd}

\newtheorem{cnt}{counter}[section]

\theoremstyle{definition}
\newtheorem{dfn}[cnt]{Definition}
\newtheorem{rem}[cnt]{Remark}
\newtheorem{bsp}[cnt]{Example}
\newtheorem{ntn}[cnt]{Notation}

\theoremstyle{plain}
\newtheorem{thm}[cnt]{Theorem}
\newtheorem{lem}[cnt]{Lemma}
\newtheorem{prp}[cnt]{Proposition}
\newtheorem{cor}[cnt]{Corollary}

\DeclareMathOperator*{\AI}{A}
\DeclareMathOperator*{\BI}{B}
\DeclareMathOperator*{\CI}{C}

\DeclareMathOperator*{\HI}{H}

\DeclareMathOperator*{\NI}{N}

\DeclareMathOperator*{\AII}{\mathcal{A}}
\DeclareMathOperator*{\BII}{\mathcal{B}}
\DeclareMathOperator*{\CII}{\mathcal{C}}
\DeclareMathOperator*{\DII}{\mathcal{D}}

\DeclareMathOperator*{\FII}{\mathcal{F}}

\DeclareMathOperator*{\HII}{\mathcal{H}}

\DeclareMathOperator*{\KII}{\mathcal{K}}
\DeclareMathOperator*{\LII}{\mathcal{L}}

\DeclareMathOperator*{\PII}{\mathcal{P}}

\DeclareMathOperator*{\RII}{\mathcal{R}}
\DeclareMathOperator*{\SII}{\mathcal{S}}

\DeclareMathOperator*{\UII}{\mathcal{U}}

\DeclareMathOperator*{\Ad}{Ad}

\DeclareMathOperator*{\s}{s}
\DeclareMathOperator*{\bds}{bds}
\DeclareMathOperator*{\w}{w}
\DeclareMathOperator*{\bdw}{bdw}
\DeclareMathOperator*{\sr}{sr}

\DeclareMathOperator*{\dom}{dom}
\DeclareMathOperator*{\im}{im}
\DeclareMathOperator*{\supp}{supp}
\DeclareMathOperator*{\suppc}{supp^{c}}
\DeclareMathOperator*{\suppcN}{supp_{\textit{N}}^{c}}
\DeclareMathOperator*{\UBII}{\UII\hspace{-0.075cm}\BII}

\DeclareMathOperator*{\esssup}{ess\,sup}

\DeclareMathOperator*{\RE}{Re\hspace{0.00785cm}}
\DeclareMathOperator*{\IM}{Im}

\DeclareMathOperator*{\rsl}{rsl\hspace{0.025cm}}
\DeclareMathOperator*{\spec}{spec\hspace{0.025cm}}
\DeclareMathOperator*{\specA}{spec_{\textit{A}}}
\DeclareMathOperator*{\specB}{spec_{\textit{B}}}
\DeclareMathOperator*{\specM}{spec_{\textit{M}}}
\DeclareMathOperator*{\specN}{spec_{\textit{N}}}
\DeclareMathOperator*{\specNnull}{spec_{\textit{N}_{0}}}
\DeclareMathOperator*{\specNone}{spec_{\textit{N}_{1}}}
\DeclareMathOperator*{\specBH}{spec_{\BII(\textit{H})}}

\newcommand{\blank}{\hspace{0.05cm}{-}}
\DeclareMathOperator*{\Cliff}{Cliff}
\DeclareMathOperator*{\If}{if}
\DeclareMathOperator*{\Else}{else}
\DeclareMathOperator*{\GL}{GL\hspace{0.025cm}}
\DeclareMathOperator*{\Inn}{Inn\hspace{0.025cm}}
\DeclareMathOperator*{\Int}{Int}
\DeclareMathOperator*{\Rj}{R_{\textit{j}}}
\DeclareMathOperator*{\sgn}{sgn}
\DeclareMathOperator*{\tr}{tr}
\DeclareMathOperator*{\vol}{\hspace{-0.05cm}\vstretch{1}{\vert}vol\vstretch{1}{\vert}}

\DeclareMathOperator*{\com}{com}
\DeclareMathOperator*{\comAj}{com_{\textit{A}_{\textit{j}}}}
\DeclareMathOperator*{\comAjperp}{com_{\textit{A}_{\textit{j}}^{\perp}}}
\DeclareMathOperator*{\comN}{com_{\textit{N}}}
\DeclareMathOperator*{\comV}{com_{\textit{V}}}
\DeclareMathOperator*{\comVperp}{com_{\textit{V}^{\perp}}}
\DeclareMathOperator*{\comVnull}{com_{\textit{V}_{0}}}
\DeclareMathOperator*{\comVone}{com_{\textit{V}_{1}}}
\DeclareMathOperator*{\comW}{com_{\textit{W}}}
\DeclareMathOperator*{\comp}{com_{\textit{p}}}
\DeclareMathOperator*{\comunit}{com_{1_{\textit{N}}}}
\DeclareMathOperator*{\comunitperp}{com_{1_{\textit{N}}^{\perp}}}
\DeclareMathOperator*{\comunitunit}{com_{1_{\textit{N}}\otimes 1_{\textit{N}}}}

\DeclareMathOperator*{\id}{id}
\DeclareMathOperator*{\inc}{inc}
\DeclareMathOperator*{\incj}{inc_{\textit{j}}}
\DeclareMathOperator*{\inck}{inc_{\textit{k}}}
\DeclareMathOperator*{\inckj}{inc_{\textit{kj}}}
\DeclareMathOperator*{\incp}{inc_{\textit{p}}}
\DeclareMathOperator*{\incq}{inc_{\textit{q}}}
\DeclareMathOperator*{\res}{res}
\DeclareMathOperator*{\resj}{res_{\textit{j}}}
\DeclareMathOperator*{\resk}{res_{\textit{k}}}
\DeclareMathOperator*{\resjk}{res_{\textit{jk}}}

\DeclareMathOperator*{\Adj}{Adj}

\DeclareMathOperator*{\op}{op}

\DeclareMathOperator*{\SIIp}{\mathcal{S}_{\textit{p}}}
\DeclareMathOperator*{\SIIq}{\mathcal{S}_{\textit{q}}}

\DeclareMathOperator*{\Fix}{Fix}
\DeclareMathOperator*{\relint}{relint}

\DeclareMathOperator*{\AC}{AC}

\DeclareMathOperator*{\Adm}{Adm}
\DeclareMathOperator*{\Admab}{Adm^{[\textit{a},\textit{b}]}}
\DeclareMathOperator*{\Admcd}{Adm^{[\textit{c},\textit{d}]}}
\DeclareMathOperator*{\Admabj}{Adm_{\textit{j}}^{[\textit{a,b}]}}
\DeclareMathOperator*{\Admabk}{Adm_{\textit{k}}^{[\textit{a,b}]}}
\DeclareMathOperator*{\Admnullone}{Adm^{[0,1]}}

\DeclareMathOperator*{\Geo}{Geo}
\DeclareMathOperator*{\Geoj}{Geo_{\textit{j}}}
\DeclareMathOperator*{\Geok}{Geo_{\textit{k}}}

\DeclareMathOperator*{\Ent}{Ent}
\DeclareMathOperator*{\Enttau}{Ent^{\tau}}
\DeclareMathOperator*{\Enttauj}{Ent^{\tau_{\textit{j}}}}

\DeclareMathOperator*{\grad}{grad}
\DeclareMathOperator*{\Hess}{Hess}
\DeclareMathOperator*{\var}{var}
\DeclareMathOperator*{\BE}{BE}
\DeclareMathOperator*{\CNV}{CNV}
\DeclareMathOperator*{\EVI}{EVI}

\DeclareMathOperator*{\Ric}{Ric\hspace{0.025cm}}
\DeclareMathOperator*{\Ilog}{I^{\log}}
\DeclareMathOperator*{\HWI}{HWI}
\DeclareMathOperator*{\MLSI}{MLSI}
\DeclareMathOperator*{\TW}{TW}

\newenvironment*{reapply}{}{}

\let\originalleft\left
\let\originalright\right
\renewcommand{\left}{\mathopen{}\mathclose\bgroup\originalleft}
\renewcommand{\right}{\aftergroup\egroup\originalright}

\DeclareMathOperator*{\lgl}{\langle}
\DeclareMathOperator*{\rgl}{\rangle}

\BeforeBeginEnvironment{align}{\renewcommand{\lgl}{\big\langle}}
\BeforeBeginEnvironment{align}{\renewcommand{\rgl}{\big\rangle}}
\AfterEndEnvironment{align}{\renewcommand{\lgl}{\langle}}
\AfterEndEnvironment{align}{\renewcommand{\rgl}{\rangle}}

\BeforeBeginEnvironment{align*}{\renewcommand{\lgl}{\big\langle}}
\BeforeBeginEnvironment{align*}{\renewcommand{\rgl}{\big\rangle}}
\AfterEndEnvironment{align*}{\renewcommand{\lgl}{\langle}}
\AfterEndEnvironment{align*}{\renewcommand{\rgl}{\rangle}}

\BeforeBeginEnvironment{itemize}{\renewcommand{\lgl}{\big\langle}}
\BeforeBeginEnvironment{itemize}{\renewcommand{\rgl}{\big\rangle}}
\AfterEndEnvironment{itemize}{\renewcommand{\lgl}{\langle}}
\AfterEndEnvironment{itemize}{\renewcommand{\rgl}{\rangle}}

\AfterEndEnvironment{reapply}{\renewcommand{\lgl}{\big\langle}}
\AfterEndEnvironment{reapply}{\renewcommand{\rgl}{\big\rangle}}

\DeclareMathOperator*{\dblv}{\|}

\BeforeBeginEnvironment{align}{\renewcommand{\dblv}{\big\|}}
\AfterEndEnvironment{align}{\renewcommand{\dblv}{\|}}

\BeforeBeginEnvironment{align*}{\renewcommand{\dblv}{\big\|}}
\AfterEndEnvironment{align*}{\renewcommand{\dblv}{\|}}

\usepackage[hyperfootnotes=false]{hyperref}

\usepackage[nobysame, msc-links]{amsrefs}

\begin{document}

\newcommand{\lc}{(}
\newcommand{\rc}{)}

\BeforeBeginEnvironment{align}{\renewcommand{\lc}{\left(}}
\BeforeBeginEnvironment{align}{\renewcommand{\rc}{\right)}}
\AfterEndEnvironment{align}{\renewcommand{\lc}{(}}
\AfterEndEnvironment{align}{\renewcommand{\rc}{)}}

\BeforeBeginEnvironment{align*}{\renewcommand{\lc}{\left(}}
\BeforeBeginEnvironment{align*}{\renewcommand{\rc}{\right)}}
\AfterEndEnvironment{align*}{\renewcommand{\lc}{(}}
\AfterEndEnvironment{align*}{\renewcommand{\rc}{)}}

\BeforeBeginEnvironment{equation}{\renewcommand{\lc}{\left(}}
\BeforeBeginEnvironment{equation}{\renewcommand{\rc}{\right)}}
\AfterEndEnvironment{equation}{\renewcommand{\lc}{(}}
\AfterEndEnvironment{equation}{\renewcommand{\rc}{)}}

\BeforeBeginEnvironment{equation*}{\renewcommand{\lc}{\left(}}
\BeforeBeginEnvironment{equation*}{\renewcommand{\rc}{\right)}}
\AfterEndEnvironment{equation*}{\renewcommand{\lc}{(}}
\AfterEndEnvironment{equation*}{\renewcommand{\rc}{)}}

\BeforeBeginEnvironment{itemize}{\renewcommand{\lc}{\left(}}
\BeforeBeginEnvironment{itemize}{\renewcommand{\rc}{\right)}}
\AfterEndEnvironment{itemize}{\renewcommand{\lc}{(}}
\AfterEndEnvironment{itemize}{\renewcommand{\rc}{)}}

\AfterEndEnvironment{reapply}{\renewcommand{\lc}{\left(}}
\AfterEndEnvironment{reapply}{\renewcommand{\rc}{\right)}}

\newcommand{\lb}{[}
\newcommand{\rb}{]}

\BeforeBeginEnvironment{align}{\renewcommand{\lb}{\left[}}
\BeforeBeginEnvironment{align}{\renewcommand{\rb}{\right]}}
\AfterEndEnvironment{align}{\renewcommand{\lb}{[}}
\AfterEndEnvironment{align}{\renewcommand{\rb}{]}}

\BeforeBeginEnvironment{align*}{\renewcommand{\lb}{\left[}}
\BeforeBeginEnvironment{align*}{\renewcommand{\rb}{\right]}}
\AfterEndEnvironment{align*}{\renewcommand{\lb}{[}}
\AfterEndEnvironment{align*}{\renewcommand{\rb}{]}}

\BeforeBeginEnvironment{equation}{\renewcommand{\lb}{\left[}}
\BeforeBeginEnvironment{equation}{\renewcommand{\rb}{\right]}}
\AfterEndEnvironment{equation}{\renewcommand{\lb}{[}}
\AfterEndEnvironment{equation}{\renewcommand{\rb}{]}}

\BeforeBeginEnvironment{equation*}{\renewcommand{\lb}{\left[}}
\BeforeBeginEnvironment{equation*}{\renewcommand{\rb}{\right]}}
\AfterEndEnvironment{equation*}{\renewcommand{\lb}{[}}
\AfterEndEnvironment{equation*}{\renewcommand{\rb}{]}}

\BeforeBeginEnvironment{itemize}{\renewcommand{\lb}{\left[}}
\BeforeBeginEnvironment{itemize}{\renewcommand{\rb}{\right]}}
\AfterEndEnvironment{itemize}{\renewcommand{\lb}{[}}
\AfterEndEnvironment{itemize}{\renewcommand{\rb}{]}}

\AfterEndEnvironment{reapply}{\renewcommand{\lb}{\left[}}
\AfterEndEnvironment{reapply}{\renewcommand{\rb}{\right]}}

\newcommand{\lset}{\vstretch{1.125}{\{}\hspace{0.007125cm}}
\newcommand{\rset}{\hspace{0.007125cm}\vstretch{1.125}{\}}}
\newcommand{\vset}{\vert}

\BeforeBeginEnvironment{align}{\renewcommand{\lset}{\bigg\{\hspace{0.025cm}}}
\BeforeBeginEnvironment{align}{\renewcommand{\rset}{\hspace{0.025cm}\bigg\}}}
\BeforeBeginEnvironment{align}{\renewcommand{\vset}{\big\vert}}
\AfterEndEnvironment{align}{\renewcommand{\lset}{\vstretch{1.125}{\{}\hspace{0.007125cm}}}
\AfterEndEnvironment{align}{\renewcommand{\rset}{\hspace{0.007125cm}\vstretch{1.125}{\}}}}
\AfterEndEnvironment{align}{\renewcommand{\vset}{\vert}}

\BeforeBeginEnvironment{align*}{\renewcommand{\lset}{\bigg\{\hspace{0.025cm}}}
\BeforeBeginEnvironment{align*}{\renewcommand{\rset}{\hspace{0.025cm}\bigg\}}}
\BeforeBeginEnvironment{align*}{\renewcommand{\vset}{\big\vert}}
\AfterEndEnvironment{align*}{\renewcommand{\lset}{\vstretch{1.125}{\{}\hspace{0.007125cm}}}
\AfterEndEnvironment{align*}{\renewcommand{\rset}{\hspace{0.007125cm}\vstretch{1.125}{\}}}}
\AfterEndEnvironment{align*}{\renewcommand{\vset}{\vert}}

\newcommand{\bpsi}{\mathrlap{\phantom{\phi}}\psi}

\newcommand\restr[3]{{
    \mathrlap{#2}
        \left.\kern-\nulldelimiterspace
        \vstretch{#1}{\phantom{#2}}
        \right\vert_{#3}
    }}

\newcommand\absv[2]{{
    \vstretch{#1}{\vert}#2\vstretch{#1}{\vert}
    }}

\newcommand\babsv[2]{{
    \vstretch{#1}{\big\vert}#2\vstretch{#1}{\big\vert}
    }}

\newcommand\bbabsv[2]{{
    \vstretch{#1}{\bigg\vert}#2\vstretch{#1}{\bigg\vert}
    }}

\newcommand\bbbabsv[2]{{
    \vstretch{#1}{\Bigg\vert}#2\vstretch{#1}{\Bigg\vert}
    }}


\frontmatter

\title{\Large{\textbf{QUANTUM OPTIMAL TRANSPORT\\ FOR AF-$\mathbf{C^{*}}$-ALGEBRAS}}}
\author{\large{DAVID F. HORNSHAW}}
\date{}

\maketitle
\pagenumbering{gobble}


\begin{abstract}
We introduce quantum optimal transport of states on tracial AF-$C^{*}$-algebras to study non-spatial transport of quantum information, and view it as the pointwise case of a general parametrised one. We define quantum optimal transport distances as dynamic transport distances in a tracial but non-ergodic and infinite-dimensional quantum setting, called AF-$C^{*}$-setting, clearly motivated by Benamou-Brenier-type distances. Pointwise division is replaced with inverses of evaluated operator means in the sense of Kubo and Ando, i.e.~with noncommutative division operators. To this end, we initially extend quasi-entropies after Hiai and Petz to the AF-$C^{*}$-setting and use the latter to define energy functionals. We further extend foundational results of Carlen and Maas to the AF-$C^{*}$-setting and develop a theory of quantum optimal transport yielding non-spatial lower Ricci bounds suitable for meaningful geometric analysis. Essential for our discussion is a coarse graining process arising from the underlying metric geometry as encoding scheme of the given tracial AF-$C^{*}$-algebra. Since energy functionals are $\Gamma$-limits w.r.t.~the coarse graining process, the latter reduces the AF-$C^{*}$-setting to the finite-dimensional one s.t.~ergodicity is recovered up to a controlled remainder. In the logarithmic mean setting, i.e.~for all quantum $L^{2}$-Wasserstein distances, we apply the coarse graining process to show equivalence of the $\EVI_{\lambda}$-gradient flow property for quantum relative entropy, its strong geodesic $\lambda$-convexity, a, possibly infinite-dimensional, Bakry-\'Emery condition, and a Hessian lower bound condition. We then define lower Ricci bounds of our quantum gradients using any one of said equivalent conditions, give sufficient conditions for lower Ricci bounds of direct sum quantum gradients and, assuming lower Ricci bounds, derive functional inequalities $\HWI_{\lambda}$, $\MLSI_{\lambda}$ and $\TW_{\lambda}$ in the AF-$C^{*}$-setting alongside their chain of implications. Fundamental example classes give quantum optimal transport of normal states on hyperfinite factors of type I and II with both non-negative and strictly positive lower Ricci bounds. An application is given by first and second quantisation of spectral triples. Upon passing to second quantisation, we introduce gauge fields as spatial coordinates in a first effort to parametrise quantum optimal transport. This yields an ansatz to study noncommutative gauge theories through the dynamics of such generalised gauge fields described as gradient flows driven by a proposed internalisation of the spectral action on gauge fields known from the celebrated spectral action principle of Connes and Chamseddine.
\end{abstract}


\chapter*{Preface}

This work fully presents the author's doctoral thesis in mathematics at the Institute for Applied Mathematics of the University of Bonn under the supervision of Karl-Theodor Sturm starting in October 2016. It was and is motivated by the lack of a general notion of curvature in Connes' program of noncommutative geometry, sufficiency of lower Ricci bounds for meaningful geometric analysis in the classical case, and, at its inception still recent, work of Carlen and Maas for lower Ricci bounds in an ergodic finite-dimensional setting using a dynamic formulation of quantum optimal transport distances. Its main goal is to extend results of Carlen and Maas, in particular their notion of lower Ricci bound based on the first properly noncommutative analogue of a classical equivalence for $\EVI_{\lambda}$-gradient flows of relative entropy, to a tracial infinite-dimensional setting in order to derive novel quantitative statements in noncommutative geometry.\par
The discussion given in this work includes such an extension to a well-behaved yet sufficiently general approximately finite-dimensional, or AF-$C^{*}$-setting. However, its exact nature, formulation and implications were not visible from the outset and thus underwent several iterations during two principal phases of work. From October 2016 to October 2020, the author worked as a member of the research group of Karl-Theodor Sturm and presented earlier versions of this work at seminars in Bonn, IST Austria, and the Oberwolfach Research Institute for Mathematics. From November 2020 onwards, he held two public-private research and development positions as applied mathematician in operations research for the German federal government and at IABG mbH.\par
The discussion which emerged during this time moves beyond the author's initial expectations and goal. The theory of quantum optimal transport as presented here lies in the intersection of noncommutative gauge theory, quantum statistical mechanics and quantum information theory. Whereas some technical effort ensures classical optimal transport theory is emulated successfully, its spatial interpretation as mass transport is invalidated by the simplest properly noncommutative example, i.e.~transport of states on two-dimensional complex matrices encoding a single qubit, since spin and mass are independent intrinsic properties of elementary particles. Any reasonable extension to the infinite-dimensional quantum setting hence requires a non-spatial interpretation as transport of, suitably general, quantum information. The author hopes to have provided the latter in this discussion, with an eye towards future applications to noncommutative gauge theory upon explicit introduction of gauge fields as spatial coordinates acting as control parameters for varying encoding schemes. An account of relations to other work at the end of the introduction focuses on the foundational work of Carlen and Maas, as well as the work of Wirth and Zhang in the tracial infinite-dimensional setting.\par


\newpage


The author's work was financed and supported by several institutions. The author gratefully acknowledges support by the European Research Council through its ERC
Advanced Grant \textit{Metric measure spaces and Ricci curvature - analytic, geometric, and probabilistic
challenges}, as well as the German Research Foundation through the SFB 1060 \textit{The Mathematics of Emergent Effects} as part of the subproject \textit{C01 - Wasserstein geometry, Ricci curvature and geometric analysis}. The author expresses gratitude to the University of Bonn and IST Austria for
supporting a research visit at IST Austria from October 2018 to January 2019, the Erwin Schr\"odinger Institute for supporting a related research visit during the workshop on \textit{Optimal Transport in Analysis and Probability} from May to June 2019, as well as the Oberwolfach Research Institute for Mathematics for the opportunity to present an earlier version of this work during the workshop on \textit{Heat Kernels, Stochastic Processes and Functional
Inequalities} in November 2019.\par
The author's work was made possible only through the advice and enduring support of several people. The author is grateful to Karl-Theodor Sturm, Matthias Erbar and Maria Gordina for their advice and support, Jan Maas for his support and invitation to visit his research group at IST Austria, as well as to Werner Ballmann and Matthias Lesch for teaching him \lc{}non-\rc{}commutative geometry during the author's many short years of study at the University of Bonn. The author further thanks Melchior Wirth for our discussions during early stages of research, Constantin Eichenberg, Lorenzo Dello Schiavo and Julian Schl\"oder for our discussions concerning mathematics, the finer points of LaTeX typesetting and time-dependent coffee selection, Thomas Rieth and Franz-Josef Schulz at IABG mbH for imparting their experience, Clemens Listner for his insights into analogue computing, as well as Michael von Prollius for his exemplary attempts to put theory into practice. The author is eternally beholden to his loving wife Marija Hornshaw, without whom this work could never have been completed, and their wonderful daughter Anne Maria Hornshaw for having made his journey worth the effort in excess of any scientific pursuits.\par

\bigskip
\bigskip
\bigskip
\bigskip

\hfill David Francis Hornshaw\par
\smallskip
\hfill Berlin, March 2024\par
\smallskip
\hfill A.M.D.G.


\newpage



\phantom{.}\vspace{-1.35cm}
\renewcommand{\contentsname}{Table of Contents}
\tableofcontents*
\addtocontents{toc}{\vspace{-0.45cm}}


\mainmatter


\chapter{Introduction}

Connes' program of noncommutative geometry \cite{BK.Con.1994.NCG}\cite{COL.Con.2021.NCG_Spectral_POV}\cite{BK.Kha.2013.NCG_Basic}\cite{BK.Kha_Mar.2008.NCG_Invitation} unifies continuous and discrete geometries \cite{BK.Gra_Var_Fig.2001.NCG_Elements}\cite{BK.vSui.2015.NCG_AF_Particle_Physics}\cite{BK.Var.2006.NCG_Elements_Short} using operator theory \cite{BK.Bla.2006.OpAlg}\cite{BK.Tak.1979.OpAlg_I}\cite{BK.Tak.2003.OpAlg_II}\cite{BK.Tak.2003.OpAlg_III}. The program lacks a general notion of curvature \cite{COL.Fat_Kha.2019.NCG_Curv_Review}\cite{COL.Les_Mos.2019.NCG_Mod_Curv_Morita_Review} even as several exist for example classes such as noncommutative tori \cite{ART.Con_Mos.2014.NCG_Mod_Curv_Tori_2D}\cite{ART.Don_Gho_Kha.2020.NCG_Tori_Ricci_Curv_3D}\cite{ART.Fat_Kha.2015.NCG_Tori_Scal_Curv_4D}\cite{ART.Les_Mos.2016.NCG_Mod_Curv_Morita}. We instead study non-spatial lower Ricci bounds, rather than curvature directly, since these often suffice for classical geometric analysis \cite{ART.Li_Yau.1986.Parabolic_Kernel_Schroedinger}\cite{BK.SalCos.2002.Aspects_Sobolev}. Lower Ricci bounds \cite{ART.Lot_Vil.2009.Classical_OT_Ricci_Bounds}\cite{ART.Stu.2006.Classical_OT_I}\cite{ART.Stu.2006.Classical_OT_II} for optimal transport on continuous geometries \cite{BK.Amb_Gig_Sav.2008.Classical_OT_GradFlow}\cite{ART.Dol_Naz_Sav.2009.Generalised_OT}\cite{BK.Vil.2009.OT} are displacement convexity \cite{ART.CorEra_McCan_Sch.2001.Displacement_Convexity_Riemannian}\cite{ART.McCan.1997.Displacement_Convexity_Local} of relative entropy. In the infinitesimally Hilbertian setting, they act as limiting cases for Bochner inequalities \cite{ART.Erb_Kuw_Stu.2015.Classical_OT_Equivalence} and imply a chain of functional inequalities \cite{ART.Lot_Vil.2009.Classical_OT_Ricci_Bounds}\cite{ART.Ott_Vil.2000.Classical_OT_LogSobolev_Talagrand} probing the underlying metric geometry. Maas \cite{ART.Maa.2011.Discrete_OT_Markov} and Mielke \cite{ART.Mie.2011.Discrete_OT_RctDiff} extended optimal transport to discrete geometries. Pointwise division is replaced with inverses of evaluated operator means in the sense of Kubo and Ando \cite{ART.And_Kub.1979.Operator_Means}. Erbar and Maas further extended lower Ricci bounds and functional inequalities \cite{ART.Erb_Fat.2018.Discrete_OT_LogSobolev}\cite{ART.Erb_Maa.2012.Discrete_OT_Ricci_Bounds}\cite{ART.Erb_Maa.2014.Discrete_OT_Porous_Medium}. Operator means let Carlen and Maas extend to an ergodic finite-dimensional quantum setting \cite{ART.Car_Maa.2014.Quantum_OT_I}\cite{ART.Car_Maa.2017.Quantum_OT_II}\cite{ART.Car_Maa.2020.Quantum_OT_III}. They allow for, possibly non-tracial, weights \cite{BK.Tak.2003.OpAlg_II}. We in turn extend their results to a tracial but non-ergodic and infinite-dimensional quantum setting, called AF-$C^{*}$-setting, and develop a theory of quantum optimal transport yielding non-spatial lower Ricci bounds suitable for meaningful geometric analysis. We in fact study a non-spatial transport of quantum information \cite{BK.Nie_Chu.2000.Quantum_Computation_Information} and view it as the pointwise case of a general parametrised one with an ansatz to study noncommutative gauge theories \cite{ART.Cha_Con.1996.NCG_Spectral_Action_I}\cite{ART.Cha_Con_vSui.2013.NCG_Inner_Fluctuations}\cite{ART.Cha_Con_vSui.2020.NCG_Second_Quantisation}\cite{BK.vSui.2015.NCG_AF_Particle_Physics}\cite{BK.Var.2006.NCG_Elements_Short}.\par
We emulate the classical case in the infinitesimally Hilbertian setting. Following work of Jordan, Kinderlehrer and Otto for Fokker-Planck equations \cite{ART.Jor_Kin_Ott.1998.Fokker_Planck}, resp.~Otto for porous medium equations \cite{ART.Ott.2001.Classical_OT_Porous_Medium}\cite{ART.Ott.2005.Classical_OT_GradFlow_DisConvex}, Ambrosio, Gigli and Savar\'e give $\EVI_{\lambda}$-gradient flows of proper l.s.c.~functionals defined on metric spaces \cite{BK.Amb_Gig_Sav.2008.Classical_OT_GradFlow} to study evolution partial differential equations using gradient flows absent differential structures \cite{ART.Dan_Sav.2008.Classical_OT_GradFlow_DisConvex}\cite{ART.Mur_Sav.2020.Classical_OT_EVI}. If $\EVI_{\lambda}$-gradient flow of relative entropy exists for $L^{2}$-Wasserstein distances determined by weak upper gradients \cite{ART.Amb_Mar_Sav.2014.Equivalence_Gradients}\cite{ART.Chee.1999.Relaxed_Gradients} inducing Dirichlet forms \cite{BK.Fuk_Osh_Tak.2011.Dirichlet_Markov}, then it is heat flow \cite{ART.Amb_Gig_Sav.2014.Classical_OT_Ricci_Bounds_I}\linebreak\cite{ART.Amb_Gig_Sav.2014.Classical_OT_Ricci_Bounds_II}. Existence is equivalent to $\lambda$-convexity of relative entropy \cite{ART.Amb_Gig_Sav.2014.Classical_OT_Ricci_Bounds_I}\cite{ART.Amb_Gig_Sav.2014.Classical_OT_Ricci_Bounds_II} and Bakry-\'Emery conditions \cite{COL.Bak_Em.1985.Hypercontractivity_Condition}\cite{ART.Bak_Led.2006.Hypercontractivity_Condition} linking heat flow to a weak Riemannian structure \cite{BK.Amb_Gig_Sav.2008.Classical_OT_GradFlow}\cite{ART.Erb.2010.Weak_Riemannian_structure} for the given classical $L^{2}$-Wasserstein distance \cite{ART.Amb_Gig_Sav.2015.Classical_OT_Ricci_Bounds_III}\cite{ART.Amb_Mon_Sav.2019.Classical_OT_Nonlinear_Diffusion}\cite{ART.Erb_Kuw_Stu.2015.Classical_OT_Equivalence}. Sturm \cite{ART.Stu.2006.Classical_OT_I}\cite{ART.Stu.2006.Classical_OT_II}, as well as Lott and Villani \cite{ART.Lot_Vil.2009.Classical_OT_Ricci_Bounds}, each established $\lambda$-convexity of relative entropy \cite{ART.CorEra_McCan_Sch.2001.Displacement_Convexity_Riemannian}\cite{ART.McCan.1997.Displacement_Convexity_Local} as synthetic lower Ricci bounds \cite{ART.Ren_Stu.2005.Smooth_OT_Equivalences}. The latter imply a $\HWI_{\lambda}$-interpolation inequality, a modified logarithmic Sobolev inequality $\MLSI_{\lambda}$, and a Talagrand inequality $\TW_{\lambda}$ \cite{ART.Lot_Vil.2009.Classical_OT_Ricci_Bounds}\cite{ART.Ott_Vil.2000.Classical_OT_LogSobolev_Talagrand}.\par


\newpage


Equivalent characterisation of heat flow as $\EVI_{\lambda}$-gradient flow of relative entropy and functional inequalities are extended to the discrete cases \cite{ART.Maa.2011.Discrete_OT_Markov}\cite{ART.Mie.2011.Discrete_OT_RctDiff} in \cite{ART.Erb_Maa.2012.Discrete_OT_Ricci_Bounds}, resp.~to the ergodic finite-dimensional setting in \cite{ART.Car_Maa.2014.Quantum_OT_I}\cite{ART.Car_Maa.2017.Quantum_OT_II}\cite{ART.Car_Maa.2020.Quantum_OT_III}. Note Datta and Rouz\'e extended results as per \cite{ART.Car_Maa.2020.Quantum_OT_III} to the finite-dimensional Lindblad setting in \cite{ART.Dat_Rou.2020.Quantum_OT_Equivalences}. In addition, see \cite{ART.Bar_Rou.2022.Quantum_Markov_Hypercontractivity}\cite{ART.Olk_Zeg.1999.NC_Hypercontractivity_Inductive}. Equivalence in \cite{ART.Car_Maa.2020.Quantum_OT_III} uses arguments fully given by Erbar and Maas in \cite{ART.Erb_Maa.2012.Discrete_OT_Ricci_Bounds} alone. The logarithmic operator mean yields analogues of $L^{2}$-Wasserstein distances and allows a Hessian lower bound condition crucial to show equivalence. In our logarithmic mean setting, which does assume the AF-$C^{*}$-setting, yet neither ergodicity nor finite trace, we extend results in \cite{ART.Car_Maa.2014.Quantum_OT_I}\cite{ART.Car_Maa.2017.Quantum_OT_II}\cite{ART.Car_Maa.2020.Quantum_OT_III} and \cite{ART.Erb_Maa.2012.Discrete_OT_Ricci_Bounds}. This demands an involved technical discussion for which we summarise our twelve main contributions as follows:

\begin{itemize}
\item[A.1)] We introduce noncommutative differential structures. They collect the data which define quantum optimal transport distances. Theorem \ref{THM.QE_QForm_Representation} and Theorem \ref{THM.NCD_Operator_Compressed_PMO} show they lets us define noncommutative division operators. They determine, and are in turn determined by, quasi-entropies in the sense of Hiai and Petz \cite{ART.Hia_Pet.2012.Quasi_Entropy_I}\cite{ART.Hia_Pet.2013.Quasi_Entropy_II} extended to the AF-$C^{*}$-setting as per Theorem \ref{THM.QE_AF}.

\item[A.2)] We define and discuss quantum optimal transport distances of states on tracial AF-$C^{*}$-algebras. These are dynamic transport distances in the AF-$C^{*}$-setting and motivated by Benamou-Brenier-type distances \cite{ART.Ben_Bre.2000.Dynamic_OT}\cite{ART.Dol_Naz_Sav.2009.Generalised_OT}. We thus define and use both quantum gradients and noncommutative division operators in our analogous constructions. Assuming traciality but allowing non-ergodicity, defined as complex kernel dimension larger than one for quantum Laplacians, we extend \cite{ART.Maa.2011.Discrete_OT_Markov}\cite{ART.Mie.2011.Discrete_OT_RctDiff} and \cite{ART.Car_Maa.2014.Quantum_OT_I}\cite{ART.Car_Maa.2017.Quantum_OT_II}\cite{ART.Car_Maa.2020.Quantum_OT_III} to the AF-$C^{*}$-setting as discussed above.

\item[A.3)] Theorem \ref{THM.QOT_Distance} shows accessibility components of quantum optimal transport distances are complete geodesic length-metric spaces \cite{BK.Amb_Gig_Sav.2008.Classical_OT_GradFlow}\cite{BK.Bur_Bur_Iv.2001.Metric_Geometry}. States at finite distance have identical fixed parts under noncommutative heat semigroups of quantum Laplacians. Non-ergodicity implies differing fixed parts. Assuming spectral gaps of quantum Laplacians and fixed parts, Theorem \ref{THM.QOT_Distance_AC_L2} classifies accessibility components of square integrable normal states using fixed parts.

\item[A.4)] We in turn use the above classification to formulate a coarse graining process as per Diagram \ref{EQ.SEC.INTRO_19}. The latter reduces the AF-$C^{*}$-setting to the finite-dimensional one s.t.~ergodicity is recovered up to a controlled remainder by reducing to accessibility components in the finite-dimensional setting. We take great care to show objects and properties are compatible with compression and finite-dimensional approximation, i.e.~restrict suitably and are scaling limits as $j\uparrow\infty$ \cite{BK.Gor_Kaz_Ott_The.2006.Coarse_Graining}.

\item[A.5)] Theorem \ref{THM.Energy_Functional_Representation} shows energy functionals are $\Gamma$-limits \cite{BK.DalMas.1993.Gamma_Convergence} w.r.t.~the coarse graining process as per Diagram \ref{EQ.SEC.INTRO_19}. We formalise the latter as existence of sufficient minimising geodesics approximated in finite dimensions. Theorem \ref{THM.QOT_Minimiser_Approximation} gives such existence. Using the latter, the coarse graining process lets us view quantum optimal transport as transport of, suitably general, quantum information. Upon allowing mixed states \cite{ART.Fri_Sio.2011.QC_Spin_Mixed_States}, we transport scaling limits of uniformly conditioned spin states encoding sequences of qubits \cite{ART.Bur_DiVi_Loss.1999.QC_Spin_QDots_as_QGates}\cite{ART.Bur_Lad_Nic.2023.QC_Spin_Overview}\cite{BK.Nie_Chu.2000.Quantum_Computation_Information}\cite{ART.DiVi.2000.Criteria}\cite{ART.DiVi_Loss.1998.Quantum_Information_Physical}.
\end{itemize}

\begin{itemize}
\item[B.1)] We extend quantum relative entropy in the sense of Araki \cite{ART.Ara.1975.Rel_Ent_I}\cite{ART.Ara.1977.Rel_Ent_II} and Umegaki \cite{ART.Ume.1962.Rel_Ent} to the AF-$C^{*}$-setting. Specifically, we extend Kosaki's formula \cite{BK.Ohy_Pet.1993.Rel_Ent} in the second variable to, possibly non-finite, traces. We require properties of the strongly unital finite-trace case. We introduce finitely supported accessibility components to rectify this. Upon restriction, Theorem \ref{THM.QOT_Distance_AC_FS} shows we recover said case as per Theorem \ref{THM.Rel_Ent_AF_Cstar_Trace} depending on the given finitely supported fixed state.

\item[B.2)] Following a maximum entropy production principle \cite{ART.Dew.2003.MaxEnt_Information_I}\cite{ART.Dew.2005.MaxEnt_Information_II}\cite{ART.Mar_Sel.2006.MaxEnt_Review}, we view quantum Laplacians as generators of quantum noise evolution. Theorem \ref{THM.L2W_Log_Mean_QNE} shows quantum Laplacians satisfy, up to sign, a quantum Fokker-Planck equation with vanishing drift term in scaling limit, i.e.~only noise diffusion term.

\item[B.3)] Theorem \ref{THM.L2W_EVI_Equivalence} yields equivalence of $\EVI_{\lambda}$-gradient flow, $\lambda$-convexity, Bakry-\'Emery and Hessian lower bound conditions by means of the coarse graining process as claimed above. We are motivated in our proof by analogous arguments in \cite{ART.Car_Maa.2020.Quantum_OT_III} and \cite{ART.Erb_Maa.2012.Discrete_OT_Ricci_Bounds}. However, Theorem \ref{THM.L2W_Log_Mean_NCDS_Fin_Hessian} replaces essential steps therein letting us argue using Riemannian metrics on relative interiors.

\item[B.4)] Lower Ricci bounds are given by $\lambda$-convexity of quantum information along minimising geodesics measured by quantum relative entropy. Their non-spatiality is further visible as follows. Assuming strictly positive lower Ricci bounds and finitely supported fixed part, Theorem \ref{THM.QOT_Distance_AC_Rel_Ent} classifies accessibility components of normal states with finite quantum relative entropy using fixed parts. Using the latter, we show strictly positive lower Ricci bounds determine energy-information trade-offs pa\-ra\-metrised by lower bounds on quantum noise.

\item[B.5)] Theorem \ref{THM.L2W_Ric} gives sufficient conditions for lower Ricci bounds of direct sum quantum gradients. In order to do so, we adapt the proof of Theorem 10.9 in \cite{ART.Car_Maa.2020.Quantum_OT_III} to the AF-$C^{*}$-setting by means of the coarse graining process. Lemma \ref{LEM.L2W_Ric} provides detailed and, to our knowledge, initially lacking proof of a necessary extension of Theorem 5 in \cite{ART.Hia_Pet.2012.Quasi_Entropy_I} to all finite-dimensional $C^{*}$-algebras.

\item[B.6)] Theorem \ref{THM.L2W_Ric_FI} derives functional inequalities $\HWI_{\lambda}$, $\MLSI_{\lambda}$ and $\TW_{\lambda}$ in the AF-$C^{*}$-setting. Non-ergodicity requires relative entropy of finitely supported fixed states in their formulation. We adapt the proofs of Theorem 11.3, Theorem 11.4 and Theorem 11.5 in \cite{ART.Car_Maa.2020.Quantum_OT_III} to the AF-$C^{*}$-setting by means of the coarse graining process. We introduce quantum Fisher information in the AF-$C^{*}$-setting.

\item[C)] We provide fundamental example classes. The latter yield quantum optimal transport of normal states on hyperfinite factors of type I and II \cite{BK.Ped.2018.Cstar_Algebras}. An application is given by first and second quantisation of spectral triples \cite{ART.Cha_Con_vSui.2013.NCG_Inner_Fluctuations}\cite{ART.Cha_Con_vSui.2020.NCG_Second_Quantisation}\cite{BK.vSui.2015.NCG_AF_Particle_Physics}\cite{BK.Var.2006.NCG_Elements_Short}. This yields our ansatz to study noncommutative gauge theories based on a proposed internalised spectral action \cite{ART.Cha_Con.1996.NCG_Spectral_Action_I}\cite{ART.Cha_Con.1997.NCG_Spectral_Action_II}\cite{ART.Cha_Con_Mar.2007.NCG_Standard_Model_Recovered}\cite{BK.vSui.2015.NCG_AF_Particle_Physics}\cite{BK.Var.2006.NCG_Elements_Short}.
\end{itemize}

\noindent The remaining introduction details $\AI.1\rc$ to $\AI.5\rc$ as per Chapter \ref{CH.NCDS} and Chapter \ref{CH.QOT}, as well as $\BI.1\rc$ to $\BI.6\rc$ as per Chapter \ref{CH.L2W}. We do not detail $\CI\rc$ here. At the end, we explain use of notation, give structure of our discussion, and elaborate on relations to other work.\par
We summarise our discussion of noncommutative differential structures given in Chapter \ref{CH.NCDS} and construction of quantum optimal transport distances as per Chapter \ref{CH.QOT}. Noncommutative differential structures collect the data which define quantum optimal transport distances. Each consists of two components and one setting. Let $(\phi,\bpsi,\gamma,\nabla)$ be such noncommutative differential structure for tracial AF-$C^{*}$-algebras $(A,\tau)$ and $(B,\omega)$ in $\lc{}f,\theta\rc$-setting. We briefly describe its components and setting necessary to establish our underlying noncommutative topology, measures and integrals.\par
The approximately finite-dimensional, or AF-$C^{*}$-algebras $A$ and $B$ are $C^{*}$-algebras s.t.~$A$ is norm closure of $A_{0}=\bigcup_{j\in\mathbb{N}}A_{j}$ and $B$ is that of $B_{0}=\bigcup_{j\in\mathbb{N}}B_{j}$ for ascending chains $\lset{}A_{j}\rset_{j\in\mathbb{N}}$ and $\lset{}B_{j}\rset_{j\in\mathbb{N}}$ of finite-dimensional $C^{*}$-algebras \cite{BK.Bla.2006.OpAlg}\cite{BK.Bro_Oza.2008.Cstar_AF}\cite{BK.Tak.1979.OpAlg_I}. Their f.s.n.~traces $\tau:A_{+}\longrightarrow [0,\infty]$ and $\omega:B_{+}\longrightarrow [0,\infty]$ are finite on $A_{0}$, resp.~$B_{0}$ \cite{BK.Dix.1977.Cstar_Algebras}\cite{BK.Tak.1979.OpAlg_I}\cite{BK.Tak.2003.OpAlg_II}. For all\linebreak $p\in [1,\infty]$, we define noncommutative $L^{p}$-spaces $L^{p}(A,\tau)$ and $L^{p}(B,\omega)$ of measurable operators equipped with $L^{p}$-norm \cite{BK.Joh_Lin.2003.Wstar_Lp_II}\cite{ART.Nel.1974.Wstar_Integration}. They fulfil H\"older inequalities. We have a modified standard pairing encoding duality \cite{BK.Tak.2003.OpAlg_II}. In particular, get $L^{\infty}(A,\tau)=L^{1}(A,\tau)^{*}$ and $L^{\infty}(B,\omega)=L^{1}(B,\omega)^{*}$. We have state space $\SII(A)=\lset\mu\in A_{+}^{*}\ \vset\ \|\mu\|_{A}=1\rset$ and normal state space $\mathcal{S}^{\NI}(A)=\SII(A)\cap L^{1}(A,\tau)^{\flat}$ of $A$. We do not require state spaces of $B$ in our discussion. We see $\tau$ and $\omega$ are, possibly unbounded \cite{BK.Pap.2002.Measures}\cite{BK.Ped.1989.Analysis_Now}, noncommutative Radon measures \lc{}cf.~Example \ref{BSP.Wstar}\rc{}. States on $A$ are noncommutative probability measures. They are normal if they have noncommutative density in $L^{1}(A,\tau)$. Elements in $B^{*}$ are noncommutative totally finite signed outer regular Radon measures \cite{BK.Pap.2002.Measures}\cite{BK.Ped.1989.Analysis_Now}.\par
We use two components in a single setting. First, we have AF-$A$-bimodule structure $(\phi,\bpsi,\gamma)$ on $B$ given by local $^{*}$-homomorphisms $\phi,\bpsi:A\longrightarrow B$ and anti-linear involution $\gamma:L^{2}(B,\omega)\longrightarrow L^{2}(B,\omega)$. AF-$C^{*}$-bimodules generalise the notion of tracial AF-$^{*}$-algebras s.t.~underlying noncommutative topologies, measures and integrals interact through local $^{*}$-homomorphisms under anti-linear involutions. We have bounded $A$-bimodule action on $B$, called the $(\phi,\bpsi)$-action, given by

\begin{align}\label{EQ.SEC.INTRO_1}
xuy=L_{x}^{\phi}\lc{}R_{y}^{\bpsi}(u)\rc{}=\phi(x)u\bpsi(y) 
\end{align}

\noindent for all $x,y\in A$ and $u\in B$. The $(\phi,\bpsi)$-action satisfies $\gamma$-symmetry given by

\begin{align}\label{EQ.SEC.INTRO_2}
\gamma\lc\phi(x)u\bpsi(y)\rc{}=\phi(y^{*})\gamma(u)\bpsi(x^{*})
\end{align}

\noindent in each case. Locality of $\phi$ and $\psi$ lets us extend Equation \ref{EQ.SEC.INTRO_1} to a normal, unital and bounded $L^{\infty}(A,\tau)$-bimodule action on $L^{2}(B,\omega)$. Moreover, Equation \ref{EQ.SEC.INTRO_2} extends in turn. We thereby see $L^{2}(B,\omega)$ is a symmetric $W^{*}$-bimodule over $L^{\infty}(A,\tau)$. This establishes, in full, noncommutative topology, measures and integrals. Secondly, we have a quantum gradient $\nabla:A_{0}\longrightarrow L^{2}(B,\omega)$. It satisfies its own locality condition. The latter shows $\nabla$ is a symmetric $W^{*}$-derivation. These are noncommutative gradients with likewise chain rule. The relationship between gradients, heat semi\-groups and Dirichlet forms extends to the noncommutative setting \cite{ART.Cip.1997.NC_Dirichlet_Markov}\cite{ART.Cip_Sav.2003.NC_Dirichlet_Grad}. We further know $\nabla\lc{}A_{0}\rc\subset B_{0}$ and $\nabla^{*}\lc{}B_{0}\rc\subset A_{0}$. Dualising $\nabla:A_{0}\longrightarrow B_{0}$ provides the weak formulation of a continuity equation as per Equation \ref{EQ.SEC.INTRO_11}. Elements in $B^{*}$ serve as synthetic tangent vectors \cite{BK.Amb_Gig_Sav.2008.Classical_OT_GradFlow}\cite{ART.Dol_Naz_Sav.2009.Generalised_OT}\cite{ART.Erb.2010.Weak_Riemannian_structure}.\par
The data collected is, by definition or construction, compatible with compression and finite-dimensional approximation by their locality properties. These are two general operations we formalise in a coarse graining process as per Diagram \ref{EQ.SEC.INTRO_19}. To this end, we\linebreak give two classes of compression used throughout our discussion.\par
We use two classes of compression. First, we compress to induced AF-$C^{*}$-bimodules. For all $j\in\mathbb{N}$, we have induced AF-$A_{j}$-bimodule structure $(\phi_{j},\bpsi_{j},\gamma_{j})=(\phi\vert_{A_{j}},\bpsi\vert_{A_{j}},\gamma\vert_{A_{j}})$ on $B_{j}$ and $j$-th restricted quantum gradient $\nabla_{\hspace{-0.055cm} j}=\nabla\vert_{A_{j}}:A_{j}\longrightarrow B_{j}$. Finite-dimensional approximation is given by $j\uparrow\infty$ for suitable convergence. Secondly, we compress with projections. Let $p\in L^{\infty}(A,\tau)$ be a projection. We have tracial $W^{*}$-algebra $L^{\infty}(A[p],\tau)=pL^{\infty}(A,\tau)p$ and symmetric $W^{*}$-bimodule $L^{2}(B[p],\omega)=pL^{2}(B,\omega)p$ over the former. We thereby compress the extended $(\phi,\bpsi)$-action with $p$ as

\begin{align}\label{EQ.SEC.INTRO_3}
xuy=L_{x,p}^{\phi}\lc{}R_{y,p}^{\bpsi}(u)\rc{}=\phi(pxp)u\bpsi(pyp)    
\end{align}

\noindent for all $x,y\in L^{\infty}(A[p],\tau)$ and $u\in L^{2}(B[p],\omega)$. Locality lets us extend Equation \ref{EQ.SEC.INTRO_3} to a unital unbounded $L^{0}(A[p],\tau)$-bimodule action on $L^{0}(B[p],\omega)$, i.e.~to their spaces of measurable operators. For all $x,y\in L^{0}(A[p],\tau)_{h}$, get joint spectral measure $E_{x,y,L^{\infty}(A[p],\tau)}$ and its domain set $\SIIp\lc{}E_{x,y}\rc$ of suitable $E_{x,y,L^{\infty}(A[p],\tau)}$-a.e.~defined $g:\mathbb{R}\times\mathbb{R}\longrightarrow\mathbb{C}$ satisfying strong resolvent convergence \cite{BK.deOli.2009.OpAlg_Quantum_Dynamics} as $\varepsilon\downarrow 0$ upon $\varepsilon$-perturbation. Each such joint spectral measure determines compressed pulled-back joint functional calculus 

\begin{align}\label{EQ.SEC.INTRO_4}
\Gamma_{x,y,p}^{L^{\phi},R^{\bpsi}}:\SIIp\lc{}E_{x,y}\rc\longrightarrow\UBII\lc{}L^{2}(B[p],\omega)\rc_{h}
\end{align}

\noindent of extended AF-$C^{*}$-bimodule actions as per Equation \ref{EQ.SEC.INTRO_3}. Note $\UBII\lc{}L^{2}(B[p],\omega)\rc$ is the set of all unbounded operators on $L^{2}(B[p],\omega)$ here. Let $A_{0,L^{\infty}(A[p],\tau)}$ be the $^{*}$-subalgebra generated by $pA_{0}p$ in $L^{\infty}(A[p],\tau)$. If $p$ satisfies additional technical properties, then we have $p$-compressed quantum gradient $\nabla_{\hspace{-0.055cm} p}=\nabla\vert_{A_{0,L^{\infty}(A[p],\tau)}}:A_{0,L^{\infty}(A[p],\tau)}\longrightarrow L^{2}(B[p],\omega)$.\par
Finally, we have a representing function $f:(0,\infty)\longrightarrow (0,\infty)$ of an operator mean \cite{ART.And_Kub.1979.Operator_Means} together with an interpolation factor $\theta\in [0,1]$ s.t.~$\|\omega\|^{1-\theta}=\omega\lc{}1_{B}\rc^{1-\theta}<\infty$. We have mean $m_{f}:(0,\infty)\times (0,\infty)\longrightarrow (0,\infty)$ given by $m_{f}(t,s)=f(ts^{-1})s$ for all $t,s>0$. For all $\varepsilon>0$, we\linebreak furthermore have mean $m_{f,\varepsilon}:[0,\infty)\longrightarrow (0,\infty)$ perturbed with $\varepsilon$ given by $m_{f,\varepsilon}(t,s)=m_{f}\lc{}t+\varepsilon,s+\varepsilon\rc$ for all $t,s\geq 0$. For all $x,y\in L^{0}(A[p],\tau)_{+}$, we have the noncommutative division operators of $x$ and $y$ given by

\begin{align}\label{EQ.SEC.INTRO_5}
\mathcal{D}_{x,y,p}^{\theta}=\Gamma_{x,y,p}^{L^{\phi},R^{\bpsi}}\lc{}m_{f}^{-\theta}\rc{}=m_{f}^{-\theta}\lc{}L_{x,p}^{\phi},R_{y,p}^{\bpsi}\rc{}
\end{align}

\noindent if $m_{f}^{-1}\in\SIIp\lc{}E_{x,y}\rc$. For all $x,y\in L^{0}(A[p],\tau)_{+}$ and $\varepsilon>0$, we also have the noncommutative division operator of $x$ and $y$ perturbed with $\varepsilon$ given by

\begin{align}\label{EQ.SEC.INTRO_6}
\mathcal{D}_{x^{\flat},y^{\flat},\varepsilon}^{\theta}=\Gamma_{x,y,p}^{L^{\phi},R^{\bpsi}}\lc{}m_{f,\varepsilon}^{-\theta}\rc{}=m_{f}^{-\theta}\lc{}L_{\mu,\varepsilon}^{\phi},R_{\eta,\varepsilon}^{\bpsi}\rc{}.
\end{align}


\pagebreak


Strong resolvent convergence as $\varepsilon\downarrow 0$ upon $\varepsilon$-perturbation is given by

\begin{align}\label{EQ.SEC.INTRO_7}
\mathcal{D}_{x,y,p}^{\theta}=\sr\textrm{-}\lim_{\varepsilon\downarrow 0}\hspace{0.025cm} \mathcal{D}_{x^{\flat},y^{\flat},\varepsilon}^{\theta}=\sr\textrm{-}\lim_{\varepsilon\downarrow 0}\hspace{0.025cm} \Gamma_{x,y,p}^{L^{\phi},R^{\bpsi}}\lc{}m_{f,\varepsilon}^{-\theta}\rc{}=\sr\textrm{-}\lim_{\varepsilon\downarrow 0}\hspace{0.025cm} m_{f,\varepsilon}^{-\theta}\lc{}L_{x,p}^{\phi},R_{y,p}^{\bpsi}\rc{}
\end{align}

\noindent if $m_{f}^{-1}\in\SIIp\lc{}E_{x,y}\rc$. This holds for applications of Equation \ref{EQ.SEC.INTRO_5} since, assuming fixed parts with integrable support, we show heat flow instantaneously regularises normal states on $A[p]$ to be, possibly unboundedly, invertible up to fixed part. States at finite distance have identical fixed parts under noncommutative heat semigroups of quantum Laplacians as per Equation \ref{EQ.SEC.INTRO_14}. We show a technical but weaker assumption on majorants of local support as per Equation \ref{EQ.SEC.INTRO_24} is stable under heat flow and ensures integrable support. The latter in turn implies suitable compressibility.\par
Equation \ref{EQ.SEC.INTRO_7} itself extends to all states on $A$ without any assumptions by means of quasi-entropies \cite{ART.Hia_Pet.2012.Quasi_Entropy_I}\cite{ART.Hia_Pet.2013.Quasi_Entropy_II}. Note quasi-entropies generalise quantum $f$-divergences \cite{ART.Hia.2018.QFD_I}\cite{ART.Hia.2019.QFD_II}, a class of dissimilarity measures for information encoded in states of quantum systems \cite{BK.Nie_Chu.2000.Quantum_Computation_Information}\cite{ART.Kra.1971.State_Changes}. We use the modified standard pairing, in particular their flat and sharp operators. For all $j\in\mathbb{N}$, we have quasi-entropy $\mathcal{I}_{j}^{f,\theta}:A_{j,+}^{*}\times A_{j,+}^{*}\times B_{j}^{*}\longrightarrow [0,\infty]$ in the finite-dimensional setting given by

\begin{align}\label{EQ.SEC.INTRO_8}
\mathcal{I}_{j}^{f,\theta}\lc\mu_{j},\eta_{j},w_{j}\rc{}=\sup_{\varepsilon>0}\hspace{0.025cm} \lgl\mathcal{D}_{\mu_{j},\eta_{j},\varepsilon}^{\theta}\lc\sharp w_{j}\rc{},\sharp w_{j}\rgl_{\omega}
\end{align}

\noindent for all $\mu,\eta\in A_{+}^{*}$ and $w\in B^{*}$. Note subscripts $j\in\mathbb{N}$ in Equation \ref{EQ.SEC.INTRO_8} denote restriction to $A_{j}$, resp.~$B_{j}$. Equation \ref{EQ.SEC.INTRO_8} uses the induced AF-$A_{j}$-bimodule structure $(\phi_{j},\bpsi_{j},\gamma_{j})$ on $B_{j}$ in each case. Monotonicity of quasi-entropies lets us extend Equation \ref{EQ.SEC.INTRO_8} as claimed to a quasi-entropy $\mathcal{I}^{f,\theta}:A_{+}^{*}\times A_{+}^{*}\times B^{*}\longrightarrow [0,\infty]$ given by

\begin{align}\label{EQ.SEC.INTRO_9}
\mathcal{I}^{f,\theta}(\mu,\eta,w)=\sup_{j\in\mathbb{N}}\hspace{0.025cm} \mathcal{I}_{j}^{f,\theta}\lc\mu_{j},\eta_{j},w_{j}\rc{}=\lim_{j\in\mathbb{N}}\hspace{0.025cm} \mathcal{I}_{j}^{f,\theta}\lc\mu_{j},\eta_{j},w_{j}\rc{}
\end{align}

\noindent for all $\mu,\eta\in A_{+}^{*}$ and $w\in B^{*}$. Equation \ref{EQ.SEC.INTRO_9} gives quasi-entropies for AF-$C^{*}$-bimodules.\par
Moreover, Equation \ref{EQ.SEC.INTRO_8} implies Equation \ref{EQ.SEC.INTRO_9} decomposes as

\begin{align}\label{EQ.SEC.INTRO_10}
\mathcal{I}^{f,\theta}(\mu,\eta,w)=\sup_{j\in\mathbb{N}}\hspace{0.025cm} \sup_{\varepsilon>0}\hspace{0.025cm} \lgl\mathcal{D}_{\mu_{j},\eta_{j},\varepsilon}^{\theta}\lc\sharp w_{j}\rc{},\sharp w_{j}\rgl_{\omega}=\sup_{\varepsilon>0}\hspace{0.025cm} \sup_{j\in\mathbb{N}}\hspace{0.025cm} \lgl\mathcal{D}_{\mu_{j},\eta_{j},\varepsilon}^{\theta}\lc\sharp w_{j}\rc{},\sharp w_{j}\rgl_{\omega}
\end{align}

\noindent in each case. Using monotonicity of nets in Equation \ref{EQ.SEC.INTRO_7} and the Kato-Robinson theorem \cite{BK.deOli.2009.OpAlg_Quantum_Dynamics}, we kill both suprema in Equation \ref{EQ.SEC.INTRO_10} by taking limits. We consequently obtain closed positive unbounded quadratic forms on $L^{2}(B,\omega)$ represented uniquely by those positive unbounded operators which extend Equation \ref{EQ.SEC.INTRO_7} to all states. Quasi-entropies as per Equation \ref{EQ.SEC.INTRO_9} define energy functionals as per Equation \ref{EQ.SEC.INTRO_12} by integrating their own evaluation on admissible paths. Altogether, we extend the quasi-entropy approach for defining noncommutative division operators in \cite{ART.Car_Maa.2020.Quantum_OT_III} to AF-$C^{*}$-bimodules.\par
We construct quantum optimal transport distances using data as above. This follows the classical case \cite{ART.Dol_Naz_Sav.2009.Generalised_OT}. Let $\overline{\SII(A)}$ denote the $w^{*}$-closure of $\SII(A)\subset A_{+}^{*}$. We metricise its $w^{*}$-topology and obtain a compact metric space. This uses separability of $A$. Note the Arzel\`{a}-Ascoli theorem applies to paths in compact metric spaces \cite{BK.Kel.1975.Topology}. For all $I=[a,b]\subset\mathbb{R}$, we have the set $\AC\lc{}I,\SII(A)\rc$ of all weakly absolutely continuous $\mu:I\longrightarrow\overline{\SII(A)}$ s.t.~$\im\mu\subset\SII(A)$. We say that $(\mu,w)\in\AC\lc{}[a,b],\SII(A)\rc\times L^{2}([a,b],B^{*})_{\w}$ is an admissible path if $(\mu,w)$ satisfies

\begin{align}\label{EQ.SEC.INTRO_11}
\frac{d}{dt}\mu(t)(x)=w(t)\lc\nabla x\rc{}=\lim_{j\in\mathbb{N}}\hspace{0.025cm} w_{j}(t)\lc\nabla_{\hspace{-0.055cm} j}x_{j}\rc{}
\end{align}

\noindent for all $x\in A_{0}$ and a.e.~$t\in [a,b]$. We call $\mu\lc{}a\rc{},\mu\lc{}b\rc\in\SII(A)$ the marginals of $(\mu,w)$, resp.~$\mu$ in this case. Note $L^{2}([a,b],B^{*})_{\w}$ is the Banach dual space of the Bochner $L^{2}$-space $L^{2}\lc{}[a,b],B\rc$, and the second identity in Equation \ref{EQ.SEC.INTRO_11} holds in general.\par
We require some bookkeeping. For all $\mu^{0},\mu^{1}\in\SII(A)$, we have the set $\Admab\lc\mu^{0},\mu^{1}\rc$ of all admissible paths defined on $[a,b]\subset\mathbb{R}$ with marginals $\mu^{0}$ and $\mu^{1}$. We further have the set $\Admab$ of all admissible paths defined on $[a,b]\subset\mathbb{R}$ regardless of marginals, as well as the set $\Adm$ of all admissible paths regardless of either definition intervals or marginals. We therefore have energy functional $E^{f,\theta}:\Adm\longrightarrow [0,\infty]$ given by

\begin{align}\label{EQ.SEC.INTRO_12}
E^{f,\theta}(\mu,w)=\int_{a}^{b}\mathcal{I}^{f,\theta}\lc\mu(t),\mu(t),w(t)\rc{}dt=\lim_{j\in\mathbb{N}}\hspace{0.025cm} \int_{a}^{b}\mathcal{I}_{j}^{f,\theta}\lc\bar{\mu}_{j}(t),\bar{\mu}_{j}(t),\bar{w}_{j}(t)\rc{}dt
\end{align}

\noindent for all $[a,b]\subset\mathbb{R}$ and $(\mu,w)\in\Admab$. Note, in contrast to Equation \ref{EQ.SEC.INTRO_11}, subscripts $j\in\mathbb{N}$ in Equation \ref{EQ.SEC.INTRO_12} denote normalised restriction to $A_{j}$, resp.~$B_{j}$ via bars. We normalise to norm one in the first two variables, and in the third one s.t.~Equation \ref{EQ.SEC.INTRO_11} remains satisfied. Since normalisation invalidates monotonicity of quasi-entropies, Equation \ref{EQ.SEC.INTRO_12} is not a supremum in general even as Equation \ref{EQ.SEC.INTRO_9} is. Upon restricting domains to sets of admissible paths with identical interval and marginals, we further show energy functionals as per Equation \ref{EQ.SEC.INTRO_12} are $\Gamma$-limits \cite{BK.DalMas.1993.Gamma_Convergence} of suitable restrictions.\par
We therefore have the quantum optimal transport distance of $(\phi,\bpsi,\gamma,\nabla)$ on $\SII(A)$ in $\lc{}f,\theta\rc$-setting given by

\begin{align}\label{EQ.SEC.INTRO_13}
\mathcal{W}_{\nabla}^{f,\theta}\lc\mu^{0},\mu^{1}\rc{}=\inf_{\Admnullone(\mu^{0},\mu^{1})}\hspace{0.025cm} \sqrt{E^{f,\theta}(\mu,w)}
\end{align}

\noindent for all $\mu^{0},\mu^{1}\in\SII(A)$. Accessibility components of quantum optimal transport distances are complete geodesic length-metric spaces. Metric geometry reduces to accessibility components. There may exist uncountable infinitely many since sets of states at finite distance have identical fixed parts under noncommutative heat semigroups of quantum Laplacians. Assuming spectral gaps of quantum Laplacians and fixed parts, we use such fixed parts to classify accessibility components of square integrable normal states. We in turn use the latter classification for the coarse graining process since its assumptions are satisfied for all accessibility components in the finite-dimensional setting.\par
Classification uses regularisation of normal states under heat flow as mentioned for Equation \ref{EQ.SEC.INTRO_7}. We have heat semigroup $h:[0,\infty)\longrightarrow\BII\lc{}L^{2}(A,\tau)\rc$ of $\Delta=\nabla^{*}\nabla$ given by

\begin{align}\label{EQ.SEC.INTRO_14}
h_{t}(u)=e^{-t\Delta}(u)
\end{align}

\noindent for all $t\geq 0$ and $u\in L^{2}(A,\tau)$. The heat semigroup of $\Delta$ extends as follows. For all $j\in\mathbb{N}$, we have symmetric $C^{*}$-derivation $\nabla_{\hspace{-0.055cm} j}:A_{j}\longrightarrow B_{j}$. We obtain $C^{*}$-Dirichlet form $u\mapsto \dblv{}\nabla_{\hspace{-0.055cm} j}u\dblv_{\tau}^{2}$ on $A_{j}$ in each case \cite{ART.Cip_Sav.2003.NC_Dirichlet_Grad}. Using the latter, we have completely Markovian semigroup $h^{j}:[0,\infty)\longrightarrow\BII(A_{j})$ as well \cite{ART.Cip.1997.NC_Dirichlet_Markov}. Note completely Markovian semigroups \cite{ART.Dav.1979.Quantum_Markov_SG_Generators}\cite{ART.Dav_Lin.1992.Noncommutative_Markov_SG_I}\cite{ART.Dav_Lin.1993.Noncommutative_Markov_SG_II} and their extensions to Banach dual spaces are given by completely positive dilations \cite{ART.Cip.1997.NC_Dirichlet_Markov}. Iterated dualisation using the modified standard pairing extends Equation \ref{EQ.SEC.INTRO_14} accordingly. Altogether, we have noncommutative heat semigroup of $\Delta$ mapping to $\BII(V)$ if $V=A^{*}$ or $V=L^{p}(A,\tau)$ for $p\in\lset{}1,2,\infty\rset$.\par
For all $\mu\in A^{*}$, $h(\mu)=h_{\infty}(\mu)$ is its fixed part and $h^{\perp}(\mu)=\mu-h(\mu)$ its image part. We call $\xi\in\SII(A)$ a fixed state, or fixed if $h(\xi)=\xi$. For all fixed states $\xi\in\SII(A)$, we have the set $\Fix_{A}(\xi)=\lset\mu\in\SII(A)\ \vset\ h(\mu)=\xi\rset$ of states on $A$ with fixed part $\xi$, as well as the set $\CII_{A}(\xi)=\lset\mu\in\SII(A)\ \vset\ \mu\sim\xi\rset$ of states on $A$ at finite distance to $\xi$. Intersecting with $\mathcal{S}^{\NI}(A)$ yields the set $\Fix_{A}^{\NI}(\xi)$, resp.~$\CII_{A}^{\NI}(\xi)$ of such normal states on $A$. These sets underpin both classification and regularisation. For all fixed states $\xi\in\SII(A)$, we have $\CII_{A}(\xi)\subset\Fix_{A}(\xi)$ and decomposition

\begin{align}\label{EQ.SEC.INTRO_15}
\textrm{Fix}_{A}(\xi)=\coprod_{\CII\subset\Fix_{A}(\xi)}\CII
\end{align}

\noindent into accessibility components. Let $\xi\in\SII(A)$ be a fixed state. We say that an accessibility component $\CII\subset (\SII(A),\mathcal{W}_{\nabla}^{f,\theta})$ has fixed part $\xi$ if $\CII\subset\Fix_{A}(\xi)$.\par
Assume $\xi\in\mathcal{S}^{\NI}(A)$ has integrable support. For all $\mu\in\Fix_{A}^{\NI}(\xi)$, we have

\begin{align}\label{EQ.SEC.INTRO_16}
h_{t}(\mu)\in\mathcal{S}_{>0}^{\NI}\big(A[\supp\xi]\big)
\end{align}

\noindent for all $t\in (0,\infty]$. Note Equation \ref{EQ.SEC.INTRO_16} uses the support projection $\supp\xi\in L^{\infty}(A,\tau)$ of $\xi$. We have $\supp\xi$-compressibility and write $A_{\xi}=A[\supp\xi]$. As such, the subscript in Equation \ref{EQ.SEC.INTRO_16} denotes normal states on $A_{\xi}$ s.t.~densities are unboundedly invertible in $\UBII\lc{}L^{2}(A_{\xi},\tau)\rc$ under compressed canonical left-~and right-action. If $\xi$, resp.~its density is boundedly invertible in this sense and square integrable, then, assuming $\Delta$ has spectral gap, regularisation as per Equation \ref{EQ.SEC.INTRO_16} lets us show

\begin{align}\label{EQ.SEC.INTRO_17}
\mathcal{C}_{A}^{\NI,2}(\xi)=\mathcal{C}_{A}(\xi)\cap\mathcal{S}^{\NI,2}(A)=\textrm{Fix}_{A}(\xi)\cap\mathcal{S}^{\NI,2}(A).
\end{align}

\noindent Upon intersecting $\Fix_{A}(\xi)$ with the set $\mathcal{S}^{\NI,2}(A)$ of all square integrable normal states on $A$, we at once see Equation \ref{EQ.SEC.INTRO_15} and Equation \ref{EQ.SEC.INTRO_17} show we classify as claimed. In the finite-dimensional setting, assumptions as above are always satisfied and we therefore classify all accessibility components using fixed parts.\par


\pagebreak


The coarse graining process as per Diagram \ref{EQ.SEC.INTRO_19} uses classification of accessibility components as per Equation \ref{EQ.SEC.INTRO_17} in the finite-dimensional setting and lets us view quantum optimal transport as transport of, suitably general, quantum information \cite{ART.Bur_Lad_Nic.2023.QC_Spin_Overview}\linebreak\cite{BK.Nie_Chu.2000.Quantum_Computation_Information}\cite{ART.DiVi_Loss.1998.Quantum_Information_Physical}. We use compression for all its vertical chains of arrows and finite-dimensional approximation for its horizontal ones. The coarse graining process decomposes global pictures, objects and properties into sequences of local ones together with a uniformity condition ensuring convergence of limits.\par
For all $j\in\mathbb{N}$, we use induced AF-$A_{j}$-bimodule structure on $B_{j}$ and $j$-th restricted quantum gradient $\nabla_{\hspace{-0.055cm} j}:A_{j}\longrightarrow B_{j}$. For all $\mu^{0},\mu^{1}\in\SII(A)$, we have

\begin{align}\label{EQ.SEC.INTRO_18}
\mathcal{W}_{\nabla}^{f,\theta}\lc\mu^{0},\mu^{1}\rc{}=\lim_{j\in\mathbb{N}}\hspace{0.025cm} \mathcal{W}_{\nabla_{\hspace{-0.055cm} j}}^{f,\theta}\big(\bar{\mu}_{j}^{0},\bar{\mu}_{j}^{1}\big).
\end{align}

\noindent Note we do have a uniformity condition as required for Equation \ref{EQ.SEC.INTRO_18} because $\mathcal{W}_{\nabla}^{f,\theta}$ is l.s.c.~in $w^{*}$-topology. In particular, we show, a priori, states are at finite distance if and only if the limit on the right-hand side of Equation \ref{EQ.SEC.INTRO_18} exists.\par
Diagram \ref{EQ.SEC.INTRO_19} itself expands the underlying process generating the limit on the right-hand side of Equation \ref{EQ.SEC.INTRO_18}. Let $j_{\min}\in\mathbb{N}$ minimal among all $j\in\mathbb{N}$ s.t.~$\xi_{j}\neq 0$. For all $j\geq j_{\min}$ in $\mathbb{N}$, we consider normalised restriction $\bar{\xi}_{j}\in\SII(A_{j})$, i.e.~a fixed state, as well $\mathcal{F}_{A_{j}}\lc\bar{\xi}_{j}\rc$ and $\mathcal{C}_{A_{j}}\lc\bar{\xi}_{j}\rc$. We require convex subset $K\subset\SII(A)$ to have lower left corner in

\begin{equation}\label{EQ.SEC.INTRO_19}
\begin{tikzcd}
A^{*}\arrow[rr, two heads] & & \ \cdots \arrow[r, two heads] & A_{j}^{*}\arrow[r, two heads] & \cdots \arrow[r, two heads] & A_{j_{\min}}^{*} \\
& & & & & \\
& & & & & \\
\mathcal{F}_{A}(\xi)\arrow[rr, two heads]\arrow[uuu,hook] & & \ \cdots \arrow[r, two heads] & \mathcal{F}_{A_{j}}\lc\bar{\xi}_{j}\rc\arrow[r, two heads]\arrow[uuu,hook] & \ \cdots \arrow[r, two heads] & \mathcal{F}_{A_{j_{\min}}}\lc\bar{\xi}_{j_{\min}}\rc\arrow[uuu,hook] \\
& & & & & \\
\mathcal{C}\cap K\arrow[rr, two heads]\arrow[uu,hook] & & \ \cdots \arrow[r, two heads] & \mathcal{C}_{A_{j}}\lc\bar{\xi}_{j}\rc\arrow[r, two heads]\arrow[uu,hook] & \cdots \arrow[r, two heads] & \mathcal{C}_{A_{j_{\min}}}\lc\bar{\xi}_{j_{\min}}\rc\arrow[uu,hook]
\end{tikzcd}
\end{equation}

\medskip

\noindent We show the AF-$C^{*}$-setting yields noncommutative analogues of scaling limits \cite{BK.Gor_Kaz_Ott_The.2006.Coarse_Graining}. As such, Diagram \ref{EQ.SEC.INTRO_19} lets us argue we transport scaling limits of uniformly conditioned spin states encoding sequences of qubits \cite{ART.Bur_DiVi_Loss.1999.QC_Spin_QDots_as_QGates}\cite{ART.Bur_Lad_Nic.2023.QC_Spin_Overview}\cite{BK.Nie_Chu.2000.Quantum_Computation_Information}\cite{ART.DiVi.2000.Criteria}\cite{ART.DiVi_Loss.1998.Quantum_Information_Physical}. Non-ergodicity restricts information-bearing degrees of freedom. Since energy functionals are $\Gamma$-limits w.r.t.~the coarse graining process, the latter reduces the AF-$C^{*}$-setting to the finite-dimensional one s.t.~ergodicity is recovered up to a controlled remainder by reducing to accessibility components in the finite-dimensional setting. If $K$ in Diagram \ref{EQ.SEC.INTRO_19} equals the domain of quantum relative entropy as per Equation \ref{EQ.SEC.INTRO_23}, then we are able to apply the coarse graining process in Chapter \ref{CH.L2W}. Altogether, we study a non-spatial transport of quantum information with restricted information-bearing degrees of freedom.\par
We describe our results in Chapter \ref{CH.L2W} for the logarithmic mean setting. We are in the latter setting if $\theta=1$ and we use the unique symmetric representing function $f=f_{\log}$ of the logarithmic operator mean $m_{\log}=m_{f_{\log}}:(0,\infty)\times (0,\infty)\longrightarrow (0,\infty)$ given by

\begin{align}\label{EQ.SEC.INTRO_20}
m_{\log}(t,s)=\frac{t-s}{\log t-\log s}=\int_{0}^{1}t^{\alpha}s^{1-\alpha}d\alpha 
\end{align}

\noindent for all $t,s>0$. We consider fixed state $\xi\in\SII(A)$ as above. We further suppress its, by assumption integrable, support projection $\supp\xi$ in all subscripts and write $\xi$ instead. If $x>0$ in $L^{\infty}(A_{\xi},\tau)$, then Equation \ref{EQ.SEC.INTRO_20} implies the noncommutative division operator of $x=y$ as per Equation \ref{EQ.SEC.INTRO_5} acts by
 
\begin{align}\label{EQ.SEC.INTRO_21}
\mathcal{D}_{x,\xi}(u)=\int_{0}^{\infty}\big(\alpha I+L_{x,p}^{\phi}\big)^{-1}\lc\big(\alpha I+R_{x,p}^{\bpsi}\big)^{-1}(u)\rc{}d\alpha
\end{align}

\noindent for all $u\in L^{2}(B_{\xi},\omega)$. Note Equation \ref{EQ.SEC.INTRO_21} corresponds to multiplication with inverses of densities in the classical case \cite{ART.Dol_Naz_Sav.2009.Generalised_OT}, resp.~use of the Kubo-Mori-Bogoliubov inner product \cite{ART.Pet_Tot.1993.Inner_Product} in \cite{ART.Car_Maa.2020.Quantum_OT_III}. As such, Equation \ref{EQ.SEC.INTRO_21} yields quantum $L^{2}$-Wasserstein distances in direct analogy to the classical case \cite{ART.Dol_Naz_Sav.2009.Generalised_OT}.\par
If $x\in L^{\infty}(A_{\xi},\tau)_{\nabla}$ s.t.~$x>0$ in $L^{\infty}(A_{\xi},\tau)$, i.e.~a boundedly invertible element in the $C^{1}$-algebra of $\nabla$ upon compressing the latter with $\supp\xi$, then $\log x\in L^{\infty}(A_{\xi},\tau)_{\nabla}$ as well and the noncommutative chain rule shows we have

\begin{align}\label{EQ.SEC.INTRO_22}
\nabla_{\hspace{-0.055cm} \xi}\log x=\mathcal{D}_{x,\xi}\nabla_{\hspace{-0.055cm} \xi}x.
\end{align}

\noindent Using results in \cite{ART.Ped.2000.OpAlg_Diff_Functions}, Equation \ref{EQ.SEC.INTRO_22} implies heat flow is, up to coarse graining, gradient flow of quantum relative entropy as per Equation \ref{EQ.SEC.INTRO_23} on relative interiors. Heat flow further satisfies a steepest entropy ascent property \cite{ART.Ber_Con_Mon.2015.MaxEnt_SEA} by considering the steepest descent property of gradient flows in smooth Riemannian manifolds \cite{BK.Lan.1995.Riemannian_Manifolds} and taking limits. We seek conditions s.t.~steepest entropy ascent implies quantum noise evolution as per $\BI.2\rc$. We accomplish this with our maximum entropy production principle \cite{ART.Dew.2003.MaxEnt_Information_I}\linebreak\cite{ART.Dew.2005.MaxEnt_Information_II}\cite{ART.Mar_Sel.2006.MaxEnt_Review}. Applying heat flow to a state for $t>0$ introduces quantum noise in $\BI.4\rc$.\par
Umegaki defined relative entropy for semi-finite $W^{*}$-algebras \cite{ART.Ume.1962.Rel_Ent}. Using relative modular operators, Araki generalised to all $W^{*}$-algebras \cite{ART.Ara.1975.Rel_Ent_I}\cite{ART.Ara.1977.Rel_Ent_II}. We extend Kosaki's formula \cite{BK.Ohy_Pet.1993.Rel_Ent} in the second variable to get the relative entropy $\Enttau:A_{+}^{*}\longrightarrow [-\infty,\infty]$ w.r.t.~$\tau$, i.e.~quantum relative entropy. It measures information required to discriminate a given state and, possibly non-finite, trace through observation. If $\mu\notin L^{1}(A,\tau)_{+}^{\flat}$, then $\mu\notin\dom\Enttau$, i.e.~$\absv{1.15}{\Ent(\mu,\tau)\hspace{0.025cm}}=\infty$ as expected. If $\mu\in L^{1}(A,\tau)_{+}^{\flat}$ and $p\in L^{1}(A,\tau)\cap L^{\infty}(A,\tau)$ is a projection s.t.~$\supp\mu\leq p$, then $\Ent(\mu,\tau)>-\infty$ and we have

\begin{align}\label{EQ.SEC.INTRO_23}
\Ent(\mu,\tau)=\sup_{\substack{n\in\mathbb{N},\\ F\in\mathcal{T}_{n}^{u}(A[p])}}\left\{\|\mu\|_{A[p]^{*}}\log n-\int_{n^{-1}}^{\infty}t^{-1}\dblv{}p-F(t)\dblv_{\mu}^{2}+t^{-2}\dblv{}F(t)\dblv_{\tau}^{2}dt\right\},
\end{align}

\noindent where we take the supremum over all suitable step functions $F:(n^{-1},\infty)\longrightarrow A[p]$ and use the GNS-inner product $\|.\|_{\mu}$ of $\mu$, resp.~$\|.\|_{\tau}$ of $\tau$ \cite{BK.Tak.1979.OpAlg_I}\cite{BK.Tak.2003.OpAlg_II}.\par


\pagebreak


The negative of Umegaki's definition is quantum entropy, i.e.~von Neumann entropy \lc{}cf.~p.17 in \cite{BK.Ohy_Pet.1993.Rel_Ent}\rc{}. Equation \ref{EQ.SEC.INTRO_23} reduces to Umegaki's definition if $\tau<\infty$. We further know it is jointly convex, l.s.c.~in $w^{*}$-topology of $L^{\infty}(A,\tau)$ and has restriction property in this case. Either may fail if $(A,\tau)$ is not strongly unital. Uniform majorisation of the local support of fixed parts suffices to prevent failure and recover a finite-dimensional approximation property. As such, we require l.s.c.~in topology of the given quantum optimal transport distance on all accessibility components with suitable fixed part, as well as compatibility with compression and finite-dimensional approximation.\par
For this, we compress with projections as per Equation \ref{EQ.SEC.INTRO_23} in general. We say that $p\in L^{1}(A,\tau)\cap L^{\infty}(A,\tau)$ majorises the local support of $\xi$ if

\begin{align}\label{EQ.SEC.INTRO_24}
\supp\xi_{j}\leq p
\end{align}

\noindent in $L^{\infty}(A,\tau)$ for a.e.~$j\in\mathbb{N}$. We further call $p$ a majorant of the local support of $\xi$. We say that $\xi$ is finitely supported if $\xi\in\dom\Enttau$ and there exists a majorant of its local support. Assume the latter. For all $\mu\in\Fix_{A}^{\NI}(\xi)$, finite-dimensional approximation is

\begin{align}\label{EQ.SEC.INTRO_25}
\Ent(\mu,\tau)=\lim_{j\in\mathbb{N}}\hspace{0.025cm} \Ent\lc\mu_{j},\tau\rc{}=\lim_{j\in\mathbb{N}}\hspace{0.025cm} \Ent\lc\bar{\mu}_{j},\tau\rc{}.
\end{align}

\noindent Finally, we show $\Enttau:\Fix_{A}^{\NI}(\xi)\longrightarrow (-\infty,\infty]$ is l.s.c.~in $\mathcal{W}_{\nabla}^{f,\theta}$-topology. We need not assume the logarithmic mean setting in our discussion of quantum relative entropy.\par
We use quantum relative entropy as measure of quantum information. Assume the logarithmic mean setting. We write $\mathcal{I}^{\log}:=\mathcal{I}^{f,1}$, as well as $E^{\log}=E^{f,1}$ and $\mathcal{W}_{\nabla}^{\log}=\mathcal{W}_{\nabla}^{f,1}$. For all $\mu^{0},\mu^{1}\in\SII(A)$, the set $\Geo\lc\mu^{0},\mu^{1}\rc$ of all minimising geodesics with marginals $\mu^{0}$ and $\mu^{1}$ is non-empty if the latter are at finite distance. Lower Ricci bounds are given by $\lambda$-convexity of quantum information as per $\CNV_{\lambda}$ below along minimising geodesics measured by quantum relative entropy. Let $\xi\in\SII(A)$ be a finitely supported fixed state. Let $\mathcal{C}\subset (\SII(A),\mathcal{W}_{\nabla}^{\log})$ be finitely supported with fixed part $\xi$ s.t.~$\mathcal{C}\cap\dom\Enttau\neq\emptyset$. Let $\lambda\in\mathbb{R}$ here. We know $\Enttau$ is $\lambda$-convex in the sense of metric geometry \cite{BK.Amb_Gig_Sav.2008.Classical_OT_GradFlow}\cite{ART.Mur_Sav.2020.Classical_OT_EVI} if for all $\mu^{0},\mu^{1}\in\mathcal{C}\cap\dom\Enttau$ and $(\mu,w)\in\Geo\lc\mu^{0},\mu^{1}\rc$ s.t.~$\mu(t)\in\dom\Enttau$ for all $t\geq 0$, we have

\begin{align}
\Ent\lc\mu(t),\tau\rc\leq \lc{}1-t\rc\Ent\lc\mu^{0},\tau\rc{}+t\Ent\lc\mu^{1},\tau\rc{}-\frac{\lambda}{2}t\lc{}1-t\rc\mathcal{W}_{\nabla}^{\log}\lc\mu^{0},\mu^{1}\rc^{2} \tag{$\CNV_{\lambda}$}
\end{align}

\noindent for all $t\in [0,1]$. We follow \cite{ART.Lot_Vil.2009.Classical_OT_Ricci_Bounds} and \cite{ART.Stu.2006.Classical_OT_I}\cite{ART.Stu.2006.Classical_OT_II}, resp.~\cite{ART.Car_Maa.2020.Quantum_OT_III}\cite{ART.Erb_Maa.2012.Discrete_OT_Ricci_Bounds} in our definition. We use $\CNV_{\lambda}$ to view lower Ricci bounds as measurement convexity of quantum information. If we have noncommutative analogues of displacement interpolations \cite{ART.CorEra_McCan_Sch.2001.Displacement_Convexity_Riemannian}\cite{ART.McCan.1997.Displacement_Convexity_Local}, then such measurement convexity in the Schr\"odinger picture is convexity under measurement of observables in the Heisenberg picture. Unfortunately, existence results are unknown to us. We instead show strictly positive lower Ricci bounds determine energy-information trade-offs pa\-ra\-metrised by lower bounds on quantum noise. Lower resolution implies lower energy paths. We avoid spatial interpretations of the classical case \cite{ART.Dol_Naz_Sav.2009.Generalised_OT}\cite{ART.Lot_Vil.2009.Classical_OT_Ricci_Bounds}.\par


\pagebreak


Strictly speaking, we apply our equivalence theorem to define lower Ricci bounds of quantum gradients in direct analogy to the classical case \cite{ART.Amb_Gig_Sav.2014.Classical_OT_Ricci_Bounds_I}\cite{ART.Amb_Gig_Sav.2014.Classical_OT_Ricci_Bounds_II}\cite{ART.Amb_Gig_Sav.2015.Classical_OT_Ricci_Bounds_III}\cite{ART.Amb_Mon_Sav.2019.Classical_OT_Nonlinear_Diffusion}\cite{ART.Erb_Kuw_Stu.2015.Classical_OT_Equivalence}, resp.~as per \cite{ART.Car_Maa.2020.Quantum_OT_III}\cite{ART.Erb_Maa.2012.Discrete_OT_Ricci_Bounds} using $\CNV_{\lambda}$ together with all of the following equivalent conditions. We see $h:[0,\infty)\times\mathcal{C}\cap\dom\Enttau\longrightarrow\mathcal{C}\cap\dom\Enttau$ is $\EVI_{\lambda}$-gradient flow of $\Enttau$ in $\mathcal{C}\cap\dom\Enttau$ in the sense of metric geometry \cite{BK.Amb_Gig_Sav.2008.Classical_OT_GradFlow}\cite{ART.Mur_Sav.2020.Classical_OT_EVI} if for all $\mu,\eta\in\mathcal{C}\cap\dom\Enttau$, we have

\begin{align}
\frac{e^{\lambda\lc{}t-s\rc{}}}{2}\mathcal{W}_{\nabla}^{\log}\lc{}h_{t}(\mu),\eta\rc^{2}-\frac{1}{2}\mathcal{W}_{\nabla}^{\log}\lc{}h_{s}(\mu),\eta\rc^{2}\leq\int_{0}^{t-s}e^{\lambda r}dr\cdot \bigg(\hspace{-0.028975cm} \Ent(\eta,\tau)-\Ent\lc{}h_{t}(\mu),\tau\rc\bigg) \tag{$\EVI_{\lambda}^{\int}$}
\end{align}

\noindent for all $0<s<t<\infty$. Note $\EVI_{\lambda}^{\int}$ as above is the well-known integral characterisation of $\EVI_{\lambda}$-gradient flows \cite{BK.Amb_Gig_Sav.2008.Classical_OT_GradFlow}, denoted by $\EVI_{\lambda}$ throughout our discussion. If $\EVI_{\lambda}$-gradient flow of relative entropy exists, then it is heat flow as above.\par
Equivalence of $\EVI_{\lambda}$ and $\CNV_{\lambda}$ is also well-known \cite{ART.Mur_Sav.2020.Classical_OT_EVI}. We have three equivalent global conditions. Upon ranging over all finitely supported accessibility components as above, the first one is $\EVI_{\lambda}$ and the second one is $\CNV_{\lambda}$. The third one is a, possibly infinite-dimensional, Bakry-\'Emery condition \cite{COL.Bak_Em.1985.Hypercontractivity_Condition}\cite{ART.Bak_Led.2006.Hypercontractivity_Condition} adapted to the logarithmic mean setting as per \cite{ART.Car_Maa.2020.Quantum_OT_III}. We say that $h$ satisfies $\BE_{\lambda}$ if for all finitely supported fixed states $\xi\in\SII(A)$ and $\mathcal{C}\subset (\SII(A),\mathcal{W}_{\nabla}^{\log})$ with fixed part $\xi$, we have

\begin{align}
\dblv{}\mathcal{M}_{\sharp\mu}^{\frac{1}{2}}\nabla h_{t}(u)\dblv_{\omega}^{2}\leq e^{-2\lambda t}\dblv{}\mathcal{M}_{h_{t}(\sharp\mu)}^{\frac{1}{2}}\nabla u\dblv_{\omega}^{2} \tag{$\BE_{\lambda}$}
\end{align}

\noindent for all $\mu\in\mathcal{C}\cap L^{2,\infty}(A_{\xi},\tau)^{\flat}$, $u\in\dom\nabla_{\hspace{-0.055cm} \xi}$ and $t\geq 0$. Note $\BE_{\lambda}$ uses those noncommutative multiplication operators whose inverses are noncommutative division operators as per Equation \ref{EQ.SEC.INTRO_5}. Compatibility with compression and finite-dimensional approximation of all objects involved, in particular but not only finite-dimensional approximation as per Equation \ref{EQ.SEC.INTRO_18} and Equation \ref{EQ.SEC.INTRO_25}, ensure all three global conditions arise from and are equivalent to three local conditions mirroring the above in the finite-dimensional setting for a.e.~induced noncommutative differential structure.\par
We therefore have $\EVI_{\lambda}$-gradient flow, $\lambda$-convexity and Bakry-\'Emery conditions in global and local form. We cannot show their equivalence directly. For this, we consider a Hessian lower bound condition $\HI_{\lambda}$ as per \cite{ART.Car_Maa.2020.Quantum_OT_III}. In the finite-dimensional logarithmic mean setting, we require such to show equivalence as claimed. We are motivated in our proof by analogous arguments in \cite{ART.Car_Maa.2020.Quantum_OT_III} and \cite{ART.Erb_Maa.2012.Discrete_OT_Ricci_Bounds}. However, we must use two differential equations for Hessians of quantum relative entropy in order to replace essential steps therein letting us argue using Riemannian metrics on relative interiors induced by the given quasi-entropy. We say that $\Hess\Enttau$ has lower bound $\lambda$ if for all for all finitely supported fixed states $\xi\in\SII(A)$ and a.e.~$j\in\mathbb{N}$ in each case, we have

\begin{align}
\Hess\hspace{-0.1525cm} \phantom{.}_{\mu}\Enttau(\eta)\geq\lambda g_{\mu}^{\bar{\xi}_{j}}(\eta,\eta) \tag{$\HI_{\lambda}$}
\end{align}

\noindent for all $\mu\in\vartheta\lc\bar{\xi}_{j}\rc$ and $\eta\in I\lc\Delta_{\bar{\xi}_{j}}\rc^{\flat}$. Each $\vartheta\lc\bar{\xi}_{j}\rc{}=\relint\mathcal{C}_{A_{j}}\lc\bar{\xi}_{j}\rc$ and $\vartheta\lc\bar{\xi}_{j}\rc\times I\lc\Delta_{\bar{\xi}_{j}}\rc^{\flat}$ is a smooth Riemannian manifold, resp.~its trivial tangent bundle plus Riemannian metric as per the right-hand side of $\HI_{\lambda}$ above. Taking limits yields equivalence as claimed.\par


\pagebreak


Following Diagram \ref{EQ.SEC.INTRO_19}, it is $\HI_{\lambda}$ which most clearly shows how underlying metric geometric properties such as lower Ricci bounds may be scaling limits of Riemannian ones up to heat flow regularised boundary. This requires suitable $K$ in Diagram \ref{EQ.SEC.INTRO_19}. Let $\xi\in\SII(A)$ be a fixed state. If $\xi\in\SII(A)$ is finitely supported fixed state, then, assuming strictly positive lower Ricci bounds, existence of a unique minimum for $\EVI_{\lambda}$-gradient flows of l.s.c.~functionals with complete sublevels \cite{ART.Mur_Sav.2020.Classical_OT_EVI} lets us show

\begin{align}\label{EQ.SEC.INTRO_26}
\mathcal{C}_{A}^{\Ent}(\xi)=\mathcal{C}_{A}(\xi)\cap\dom\Enttau=\textrm{Fix}_{A}(\xi)\cap\dom\Enttau\neq\emptyset.
\end{align}

\noindent As for Equation \ref{EQ.SEC.INTRO_17} and $K=\mathcal{S}^{\NI,2}(A)$, Equation \ref{EQ.SEC.INTRO_26} and $K=\dom\Enttau$ readily show we classify accessibility components of normal states with finite quantum relative entropy using fixed parts. This yields suitable $K$, as is visible from our equivalent conditions above. Strictly lower Ricci bounds avoid assumptions on spectral gaps. Equation \ref{EQ.SEC.INTRO_26} lets us formulate energy-information trade-offs as claimed using Talagrand inequality $\TW_{\lambda}$ for $\lambda\geq 0$ as given below. It formulates an energy-information trade-off since lower energy paths are obtained by introducing quantum noise. The latter requires our view of quantum Laplacians as generators of quantum noise evolution as per $\BI.2\rc$.\par
We then give sufficient conditions for strictly positive lower Ricci bounds of direct sum quantum gradients. We adapt the proof of Theorem 10.9 in \cite{ART.Car_Maa.2020.Quantum_OT_III} for $\lambda$-intertwining symmetric $C^{*}$-derivations to the AF-$C^{*}$-setting by means of the coarse graining process. We give an essential estimate for quasi-entropies evaluated on states under heat flow extending its analogue in \cite{ART.Car_Maa.2020.Quantum_OT_III} to the AF-$C^{*}$-setting. Our proof requires an extension of Theorem 5 in \cite{ART.Hia_Pet.2012.Quasi_Entropy_I} to all finite-dimensional $C^{*}$-algebras. Examples for strictly positive lower Ricci bounds are twisted dynamic quantum gradients induced by intertwining sets of Clifford generators. This generalises \cite{ART.Car_Maa.2014.Quantum_OT_I} but needs detailed implementation of Bogoliubov automorphisms on anti-symmetric Fock space \cite{BK.Gra_Var_Fig.2001.NCG_Elements}\cite{BK.Ply_Rob.1994.Clifford_Algebras}.\par
Assuming lower Ricci bounds, we derive functional inequalities $\HWI_{\lambda}$, $\MLSI_{\lambda}$ and $\TW_{\lambda}$ for $\lambda\geq 0$, resp.~$\lambda>0$ as per \cite{ART.Car_Maa.2020.Quantum_OT_III}. Non-ergodicity requires relative entropy of finitely supported fixed states in their formulation. We introduce quantum Fisher information in the AF-$C^{*}$-setting. Its r\^ole mirrors the classical case \cite{ART.Lot_Vil.2009.Classical_OT_Ricci_Bounds}\cite{ART.Ott_Vil.2000.Classical_OT_LogSobolev_Talagrand}. We have quantum Fisher information $\Ilog:A_{+}^{*}\longrightarrow [0,\infty]$ given by

\begin{align}\label{EQ.SEC.INTRO_27}
\Ilog(\mu)=\sup_{j\in\mathbb{N}}\hspace{0.025cm} \mathcal{I}_{j}^{\log}\lc\mu_{j},\mu_{j},\lc\nabla\sharp\mu_{j}\rc^{\flat}\rc{}
\end{align}

\noindent for all $\mu\in A_{+}^{*}$. Equation \ref{EQ.SEC.INTRO_27} immediately shows quantum Fisher information inherits properties of quasi-entropies. For all finitely supported fixed states $\xi\in\SII(A)$, we use the inherited properties and the gradient flow property to show

\begin{align}\label{EQ.SEC.INTRO_28}
\Ilog(\mu)=-\restr{0.925}{\frac{d}{dt}}{t=0}\Enttau\lc{}h_{t}(\mu)\rc{}
\end{align}

\noindent for all $\mu\in\Fix_{A}^{\NI}(\xi)\cap\mathcal{S}^{\NI}(A_{\xi})\cap\GL\lc{}L^{\infty}(A_{\xi},\tau)\rc\cap\lc\dom\Delta\rc^{\flat}$. Note $\GL\lc{}L^{\infty}(A_{\xi},\tau)\rc$ is the set of all boundedly invertible elements in $L^{\infty}(A_{\xi},\tau)$. Equation \ref{EQ.SEC.INTRO_28} implies $\Ilog(\mu)$ is indeed a noncommutative analogue for parametrisations $\lset{}h_{t}(\mu)\rset_{t\geq 0}$ given $\mu\in\mathcal{S}^{\NI}(A)$.\par


\pagebreak


We adapt the proof of Proposition 11.2 in \cite{ART.Car_Maa.2020.Quantum_OT_III} to the AF-$C^{*}$-setting by means of the coarse graining process. For all $\mu,\eta\in\SII(A)$, Equation \ref{EQ.SEC.INTRO_28} lets us show

\begin{align}\label{EQ.SEC.INTRO_29}
\limsup_{j\in\mathbb{N}}\hspace{0.025cm} \frac{d^{+}}{dt}\mathcal{W}_{\nabla}^{\log}\lc{}h_{t}\lc\bar{\mu}_{j}\rc{},\bar{\eta}_{j}\rc\leq\sqrt{\Ilog\lc{}h_{t}(\mu)\rc{}}
\end{align}

\noindent for all $t\geq 0$. Equation \ref{EQ.SEC.INTRO_29} in turn provides sufficient control of metric derivatives using quantum Fisher information. It is the crucial estimate allowing us to adapt the proofs of Theorem 11.3, Theorem 11.4 and Theorem 11.5 in \cite{ART.Car_Maa.2020.Quantum_OT_III} to the AF-$C^{*}$-setting by means of the coarse graining process.\par
We derive three functional inequalities. Let $\lambda\in\mathbb{R}$. We say that $\Enttau$ satisfies $\HWI_{\lambda}$ if for all finitely supported fixed states $\xi\in\SII(A)$ and $\mathcal{C}\subset (\SII(A),\mathcal{W}_{\nabla}^{\log})$ with fixed part $\xi$ s.t.~$\mathcal{C}\cap\dom\Enttau\neq\emptyset$, we have

\begin{align}
\Ent(\mu,\tau)\leq\mathcal{W}_{\nabla}^{\log}(\mu,\xi)\sqrt{\Ilog(\mu)}-\frac{\lambda}{2}\mathcal{W}_{\nabla}^{\log}(\mu,\xi)^{2}+\Ent(\xi,\tau) \tag{$\HWI_{\lambda}$}
\end{align} 

\noindent for all $\mu\in\mathcal{C}$. Assume $\lambda>0$. We say that $\Enttau$ satisfies $\MLSI_{\lambda}$ if for all finitely supported fixed states $\xi\in\SII(A)$ and $\mathcal{C}\subset (\SII(A),\mathcal{W}_{\nabla}^{\log})$ with fixed part $\xi$ s.t.~$\mathcal{C}\cap\dom\Enttau\neq\emptyset$, we have

\begin{align}
\Ent(\mu,\tau)\leq\frac{1}{2\lambda}\Ilog(\mu)+\Ent(\xi,\tau) \tag{$\MLSI_{\lambda}$}
\end{align}

\noindent for all $\mu\in\mathcal{C}$. We further say that $\Enttau$ satisfies $\TW_{\lambda}$ if for all finitely supported fixed states $\xi\in\SII(A)$ and $\mathcal{C}\subset (\SII(A),\mathcal{W}_{\nabla}^{\log})$ with fixed part $\xi$ s.t.~$\mathcal{C}\cap\dom\Enttau\neq\emptyset$, we have

\begin{align}
\mathcal{W}_{\nabla}^{\log}(\mu,\xi)\leq\sqrt{\frac{2}{\lambda}\lc\Ent(\mu,\tau)-\Ent(\xi,\tau)\rc{}} \tag{$\TW_{\lambda}$}
\end{align}

\noindent for all $\mu\in\mathcal{C}$. We obtain the following implications in direct analogy to the classical case \cite{ART.Lot_Vil.2009.Classical_OT_Ricci_Bounds}\cite{ART.Ott_Vil.2000.Classical_OT_LogSobolev_Talagrand} and extending results in \cite{ART.Car_Maa.2020.Quantum_OT_III} to the AF-$C^{*}$-setting as claimed. Analogous to our equivalent conditions for lower Ricci bounds, all three functional inequalities above are scaling limits w.r.t.~the coarse graining process.\par
We have an expected chain of functional inequalities. If we do have lower Ricci bounds for $\lambda\in\mathbb{R}$, then $\Enttau$ satisfies $\HWI_{\lambda}$. If $\Enttau$ satisfies $\HWI_{\lambda}$ for $\lambda>0$, then $\Enttau$ in turn satisfies $\MLSI_{\lambda}$. If $\Enttau$ satisfies $\MLSI_{\lambda}$, then $\Enttau$ finally satisfies $\TW_{\lambda}$. Their proofs pass through the finite-dimensional setting.

\medskip

\noindent\textbf{Notation.} We follow notational conventions of our stated standard references whenever possible. However, we must tie together different ones tailored to our use. We establish a single coherent notation in our definitions and paragraphs marked as \textbf{Notation}. Unless stated otherwise, the latter are in force once stated. This includes notation given in the appendix. The latter are revisited in the main matter prior to first use.

\noindent\textbf{Structure.} We divide our discussion into main matter and its appendix. The latter gives auxiliary technical results. In Chapter \ref{CH.NCDS}, we discuss the data necessary to define quantum optimal transport distances and collect such initial data in noncommutative differential structures. This covers $\AI.1\rc$. In Chapter \ref{CH.QOT}, we define our quantum optimal transport distances, discuss fundamental properties and provide fundamental example classes. This covers $\AI.2\rc$ to $\AI.5\rc$, and $\CI\rc$. In Chapter \ref{CH.L2W}, we construct quantum relative entropy for, possibly non-finite, traces, discuss the logarithmic mean setting, and extend results in \cite{ART.Car_Maa.2014.Quantum_OT_I}\cite{ART.Car_Maa.2017.Quantum_OT_II}\cite{ART.Car_Maa.2020.Quantum_OT_III} and \cite{ART.Erb_Maa.2012.Discrete_OT_Ricci_Bounds} to the AF-$C^{*}$-setting. This covers $\BI.1\rc$ to $\BI.6\rc$.

\medskip

\noindent\textbf{Relations to other work.} We may categorise noncommutative optimal transport into dynamic \cite{ART.Bre_Vor.2020.Quantum_OT_Dynamic_Full_Matrix_Measures}\cite{ART.Car_Maa.2014.Quantum_OT_I}\cite{ART.Car_Maa.2017.Quantum_OT_II}\cite{ART.Car_Maa.2020.Quantum_OT_III}\cite{ART.Che_Geo_Tan.2018.Quantum_OT_Dynamic_Full_Matrix_I}\cite{ART.Che_Geo_Tan.2018.Quantum_OT_Dynamic_Full_Matrix_II}\cite{ART.Chi_Pey_Sol_Via.2019.Quantum_OT_Dynamic_Entropic}\cite{PRE.Wir.2018.NC_OT} and static \cite{PRE.Ant_Cav.Quantum_OT_Static_Grassmannian}\cite{ART.DePa_Tre.2021.Quantum_OT_Static_Channels}\cite{ART.Duv.2022.Quantum_OT_Static_Transport_Plans_I}\cite{ART.Fel_Ger_Por.2023.Quantum_OT_Static_Entropic}\cite{ART.Col_Eck_Fri_Zyc.2022.Quantum_OT_Static_Bipartite} formulations. As explained above, quantum optimal transport distances as per $\AI.2\rc$ are dynamic transport distances motivated by Benamou-Brenier-type distances \cite{ART.Ben_Bre.2000.Dynamic_OT}\cite{ART.Dol_Naz_Sav.2009.Generalised_OT}. The latter is shared by all dynamic formulations. Following work of Maas and Mielke for the discrete cases \cite{ART.Maa.2011.Discrete_OT_Markov}\cite{ART.Mie.2011.Discrete_OT_RctDiff}, Carlen and Maas pioneered the dynamic formulation in \cite{ART.Car_Maa.2017.Quantum_OT_II}\cite{ART.Car_Maa.2020.Quantum_OT_III} to study quantum Fokker-Planck equations \cite{ART.Car_Maa.2014.Quantum_OT_I}. Our discussion, resp.~any of its prior versions, and independent but concurrent work of Wirth \cite{PRE.Wir.2018.NC_OT} together with Zhang \cite{ART.Wir_Zha.2021.Quantum_OT_Complete_Gradient_Estimates} are the first infinite-dimensional dynamic formulations and extensions of results in \cite{ART.Car_Maa.2014.Quantum_OT_I}\cite{ART.Car_Maa.2017.Quantum_OT_II}\cite{ART.Car_Maa.2020.Quantum_OT_III}. Assuming regular operator mean and restricting to densities, i.e.~normal states, the dynamic formulation in \cite{PRE.Wir.2018.NC_OT} and our discussion coincide. However, each has considerably different technical foundation, assumptions and applicability. Results and their proofs, as well as range of examples, differ accordingly. We closely examine these differences further below. In general terms, we see our bottom-up design yields flexible architecture for the AF-$C^{*}$-setting capable of stronger results therein.\par
All dynamic formulations avoid the lack of a natural noncommutative analogue of conditioning for couplings \cite{ART.Car_Leb_Lie.2013.Quantum_OT_Static_Extension_Problem}. Recent static ones consider specific sets of couplings or quantum channels for trace-class operators on Hilbert spaces \cite{PRE.Ant_Cav.Quantum_OT_Static_Grassmannian}\cite{ART.DePa_Tre.2021.Quantum_OT_Static_Channels}\cite{ART.Col_Eck_Fri_Zyc.2022.Quantum_OT_Static_Bipartite}, balanced transport plans \cite{ART.Duv.2022.Quantum_OT_Static_Transport_Plans_I}\cite{PRE.Duv_Sko_Sny.2022.Quantum_OT_Static_Transport_Plans_II}, or use entropic regularisation \cite{ART.Fel_Ger_Por.2023.Quantum_OT_Static_Entropic}. Noncommutative duality formulas remain difficult to find. Following the work of Erbar, Maas and Wirth \cite{ART.Erb_Maa_Wir.2019.Discrete_OT_Dual} and Gangbo, Li and Mou \cite{ART.Gan_Li_Mou.2019.Discrete_OT_Dual} for discrete cases, Wirth gives such a duality formula \cite{ART.Wir.2021.Quantum_OT_Dual} for quantum optimal transport distances in the finite-dimensional setting, resp.~their entropic regularisations \cite{ART.Beck_Li.2021.Quantum_OT_Dynamic_Entropic_Stat_Learning}, via subsolutions of Hamilton-Jacobi-Bellmann equations \cite{ART.Bob_Gen_Led.2001.HJB_Hypercontracitivity}\cite{ART.Ott_Vil.2000.Classical_OT_LogSobolev_Talagrand}. We do not have an infinite-dimensional extension but consider finding one by means of the coarse graining process a test of our approach we defer to future work.\par
We focus on the relation of our main contributions to the two most related dynamic\linebreak formulations \cite{ART.Car_Maa.2014.Quantum_OT_I}\cite{ART.Car_Maa.2017.Quantum_OT_II}\cite{ART.Car_Maa.2020.Quantum_OT_III} and \cite{PRE.Wir.2018.NC_OT}. Moreover, we consider the use of our discussion for studying noncommutative gauge theories \cite{ART.Cha_Con.1996.NCG_Spectral_Action_I}\cite{ART.Cha_Con_vSui.2013.NCG_Inner_Fluctuations}\cite{ART.Cha_Con_vSui.2020.NCG_Second_Quantisation}\cite{BK.vSui.2015.NCG_AF_Particle_Physics}\cite{BK.Var.2006.NCG_Elements_Short} within Connes' program of noncommutative geometry \cite{BK.Con.1994.NCG}\cite{COL.Con.2021.NCG_Spectral_POV}\cite{BK.Kha.2013.NCG_Basic}\cite{BK.Kha_Mar.2008.NCG_Invitation}. This leads us to view quantum optimal transport as transport of quantum information \cite{BK.Nie_Chu.2000.Quantum_Computation_Information} without considering spatial coordinates \cite{ART.Con.1996.NCG_Reconstruction}. Furthermore, we view quantum Laplacians as generators of quantum noise evolution as per $\BI.2\rc$ in order to have non-spatiality of lower Ricci bounds as per $\BI.4\rc$ and associated energy-information trade-offs. Applications of other approaches to entropic inequalities \cite{ART.Amb_Gio_DePa_Tre.2018.Quantum_Information_Gaussian_Optimizer_Inequalities}, quantum channels \cite{ART.DePa_Tre.2021.Quantum_OT_Static_Channels}\cite{ART.Duv.2023.Quantum_OT_Static_Channels_L1}\cite{PRE.Gao_Rou.2021.Quantum_OT_Static_Channels_Ricci_Curvature}, statistical learning \cite{ART.Beck_Li.2021.Quantum_OT_Dynamic_Entropic_Stat_Learning} and variational algorithms \cite{ART.Fra_Mar_DePa_Rou.2023.Quantum_OT_Algorithms_Variation}, among several more \cite{ART.Chi_Pey_Sol_Via.2019.Quantum_OT_Dynamic_Entropic}\cite{ART.Dat_Rou.2019.Quantum_OT_Concentration_Inequalities}\cite{ART.Dat_Rou.2020.Quantum_OT_Equivalences}, fit our point of view.\par


\newpage


We explain similarities and differences of foundational work of Carlen and Maas \cite{ART.Car_Maa.2014.Quantum_OT_I}\cite{ART.Car_Maa.2017.Quantum_OT_II}\cite{ART.Car_Maa.2020.Quantum_OT_III}, as well as related work of Wirth \cite{PRE.Wir.2018.NC_OT}, resp.~Wirth and Zhang \cite{ART.Wir_Zha.2021.Quantum_OT_Complete_Gradient_Estimates}, to our discussion. As explained in the introduction, we extend results in \cite{ART.Car_Maa.2014.Quantum_OT_I}\cite{ART.Car_Maa.2017.Quantum_OT_II}\cite{ART.Car_Maa.2020.Quantum_OT_III} and \cite{ART.Erb_Maa.2012.Discrete_OT_Ricci_Bounds} to the AF-$C^{*}$-setting as per $\AI.1\rc$ by means of the coarse graining process as per $\AI.4\rc$. Assumptions differ from ours in two points apart from dimensionality. First, they allow for, possibly non-tracial, weights \cite{BK.Tak.2003.OpAlg_II}. Of course, they are finite. We assume traciality but not finiteness. Traciality implies neither our discussion nor \cite{PRE.Wir.2018.NC_OT} fully subsumes \cite{ART.Car_Maa.2020.Quantum_OT_III}. Secondly, they assume ergodicity and we do not. Using our assumptions, which let us cover all fundamental example classes as per $\CI\rc$, we extend the equivalence in \cite{ART.Car_Maa.2020.Quantum_OT_III} to that of the $\EVI_{\lambda}$-gradient flow of quantum relative entropy as per $\BI.1\rc$, its strong geodesic $\lambda$-convexity, a, possibly infinite-dimensional, Bakry-\'Emery condition, and a Hessian lower bound condition as per $\BI.3\rc$. This is our equivalence theorem. We further obtain non-spatial lower Ricci bounds as per $\BI.4\rc$, sufficient conditions for lower Ricci bounds of direct sum quantum gradients as per $\BI.5\rc$, and derive functional inequalities $\HWI_{\lambda}$, $\MLSI_{\lambda}$ and $\TW_{\lambda}$ as per $\BI.6\rc$. Finite-dimensional cases are given in \cite{ART.Car_Maa.2020.Quantum_OT_III}.\par
Yet we cannot naively extend results to the AF-$C^{*}$-setting by taking limits. We thus consider objects and properties compatible with compression and finite-dimensional approximation as per $\AI.4\rc$. We explain such compatibility at the end of Chapter \ref{CH.NCDS} and formalise it in the coarse graining process \cite{BK.Gor_Kaz_Ott_The.2006.Coarse_Graining} in Chapter \ref{CH.QOT}. This in turn demands an involved technical discussion culminating in our introduction of finitely supported accessibility components as per $\BI.1\rc$ and our restriction of quantum relative entropy to the latter. An essential technique is compressed pulled-back joint functional calculus of extended AF-$C^{*}$-bimodule actions explained in Chapter \ref{CH.NCDS} based on Appendix \ref{APP.A} and Appendix \ref{APP.B}. However, use of the coarse graining process requires us to adapt or even replace essential arguments in \cite{ART.Car_Maa.2020.Quantum_OT_III} and \cite{ART.Erb_Maa.2012.Discrete_OT_Ricci_Bounds} as explained for our main contributions and throughout our discussion when proving suitable analogous results.\par
Wirth gives a dynamic formulation \cite{PRE.Wir.2018.NC_OT} in a tracial infinite-dimensional setting. Assuming energy dominant trace \cite{ART.Jun_Zen.2015.Energy_Dominant_Trace} but not ergodicity, these are noncommutative optimal transport distances of densities, i.e.~normal states, in tracial $W^{*}$-algebras. They are determined by suitable symmetric $C^{*}$-derivations inducing $C^{*}$-Dirichlet forms on noncommutative $L^{2}$-spaces of tracial $W^{*}$-algebras \cite{ART.Cip.1997.NC_Dirichlet_Markov}\cite{ART.Cip_Sav.2003.NC_Dirichlet_Grad}. Results in \cite{PRE.Wir.2018.NC_OT}\cite{ART.Wir_Zha.2021.Quantum_OT_Complete_Gradient_Estimates} often assume tracial state and may assume ergodicity. Note \cite{ART.Wir_Zha.2021.Quantum_OT_Complete_Gradient_Estimates} is based on \cite{PRE.Wir.2018.NC_OT}. Assuming tracial state and ergodicity, Wirth shows a, possibly infinite-dimensional, Bakry-\'Emery condition \cite{PRE.Wir.2018.NC_OT} as per \cite{ART.Car_Maa.2020.Quantum_OT_III} implies heat flow is $\EVI_{\lambda}$-gradient flow of relative entropy for $W^{*}$-algebras \cite{BK.Ohy_Pet.1993.Rel_Ent} and therefore, by standard arguments \cite{ART.Mur_Sav.2020.Classical_OT_EVI}, $\lambda$-convexity of such relative entropy. This cannot satisfyingly define lower Ricci bounds since \cite{PRE.Wir.2018.NC_OT} lacks full equivalence as per $\BI.3\rc$. Assuming tracial state, Wirth and Zhang give sufficient conditions for satisfying Bakry-\'Emery conditions \cite{ART.Wir_Zha.2021.Quantum_OT_Complete_Gradient_Estimates} as per \cite{ART.Car_Maa.2020.Quantum_OT_III} using intertwining property for general families of bounded linear operators. They need not assume direct sum noncommutative gradients since they give an argument dual to the monotonicity argument in \cite{ART.Car_Maa.2020.Quantum_OT_III} we extend. Assuming tracial state, Wirth and Zhang obtain functional inequalities $\HWI_{\lambda}$, $\MLSI_{\lambda}$ \cite{ART.Wir_Zha.2021.Quantum_OT_Complete_Gradient_Estimates} and $\TW_{\lambda}$ \cite{PRE.Wir.2018.NC_OT} as per \cite{ART.Car_Maa.2020.Quantum_OT_III} using relative entropy for $W^{*}$-algebras conditioned to fixed-point subalgebras. Such a priori conditioning handles non-ergodicity but does not emerge from an underlying metric geometry.\par
Note \cite{PRE.Wir.2018.NC_OT}\cite{ART.Wir_Zha.2021.Quantum_OT_Complete_Gradient_Estimates} and our discussion share the tracial infinite-dimensional setting. Yet each approach has considerably different technical foundation, assumptions and applicability. Results and their proofs, as well as range of examples, differ accordingly. We examine these differences. Whereas noncommutative differential structures collect our initial data, \cite{PRE.Wir.2018.NC_OT} considers $C^{*}$-Dirichlet forms \cite{ART.Alb_Ser.1977.Cstar_Dirichlet_Markov} in order to define test algebras of observables via Lipschitz seminorms using the induced noncommutative gradient \cite{ART.Cip.1997.NC_Dirichlet_Markov}\linebreak\cite{ART.Cip_Sav.2003.NC_Dirichlet_Grad} and given operator mean \cite{ART.And_Kub.1979.Operator_Means}. They may equivalently assume a given symmetric $C^{*}$-derivation, i.e.~noncommutative gradient, as we do using quantum gradients. If both approaches apply, then test algebras in \cite{PRE.Wir.2018.NC_OT} are larger and contain ours, i.e.~unions of all generating $C^{*}$-subalgebras. Assuming regular operator mean and restricting to densities, the dynamic formulation in \cite{PRE.Wir.2018.NC_OT} and our discussion coincide for two further reasons. First, \cite{PRE.Wir.2018.NC_OT} assumes energy dominant trace in order to have $\sigma$-weak extensions of bimodule actions. We show a general extension of AF-$C^{*}$-bimodule actions to spaces of measurable operators using extendability of local $^{*}$-homomorphisms. Note we thereby avoid use of $C^{*}$-Dirichlet forms as in \cite{ART.Jun_Zen.2015.Energy_Dominant_Trace}\cite{PRE.Wir.2018.NC_OT}. Secondly, \cite{PRE.Wir.2018.NC_OT} uses noncommutative multiplication operators for densities. We construct noncommutative division operators for all states. We show both choices are equivalent in the finite-dimensional setting by considering vector fields along admissible paths minimising the given quasi-entropy at a.e.~time. Taking limits shows both dynamic formulations coincide as claimed.\par
Assumptions and applicability differ from ours in several points. Results in \cite{PRE.Wir.2018.NC_OT}\linebreak\cite{ART.Wir_Zha.2021.Quantum_OT_Complete_Gradient_Estimates} often assume tracial state and may assume ergodicity. As stated above, we assume neither. We give three differences to \cite{PRE.Wir.2018.NC_OT} and two to \cite{ART.Wir_Zha.2021.Quantum_OT_Complete_Gradient_Estimates} showing why their results are insufficient for our purposes. First, they have weaker results concerning existence of minimising geodesics. Assuming tracial state, the logarithmic mean setting and heat flow is $\EVI_{\lambda}$-gradient flow of relative entropy for $W^{*}$-algebras, \cite{PRE.Wir.2018.NC_OT} shows existence of minimising geodesics for densities at finite distance in the domain of relative entropy for $W^{*}$-algebras. We show each accessibility component for any given symmetric operator mean is a geodesic length-metric space s.t.~minimising geodesics approximated in finite dimensions as per $\AI.5\rc$ exist between states at finite distance. This only requires our initial data. Secondly, they lack classification of accessibility components. We show two such classifications for varying assumptions on states as per $\AI.3\rc$ and $\BI.4\rc$. These coincide in the finite-dimensional logarithmic mean setting. Thirdly, they do not prove an equivalence theorem as per \cite{ART.Car_Maa.2020.Quantum_OT_III} or $\BI.3\rc$. Assuming tracial state and ergodicity, \cite{PRE.Wir.2018.NC_OT} shows the chain of implications starting from a Bakry-\'Emery condition stated above.\par
We assume neither and prove full equivalence as per $\BI.3\rc$. We state and prove such using existence of sufficient minimising geodesics approximated in finite dimensions and classification. We use the latter for the coarse graining process and our control of quantum relative entropy as per $\BI.1\rc$ on finitely supported accessibility components. In contrast, tracial states are necessary in \cite{PRE.Wir.2018.NC_OT} for existence of geodesics and since it gives no extension of relative entropy for $W^{*}$-algebras as per $\BI.1\rc$. The proof of equivalence in \cite{ART.Car_Maa.2020.Quantum_OT_III} uses direct calculations involving the Hessian of quantum relative entropy in the finite-dimensional Riemannian setting. We engage in our own and apply the coarse graining process. We see no substitute for this in \cite{PRE.Wir.2018.NC_OT}, resp.~its continued development in \cite{ART.Wir_Zha.2021.Quantum_OT_Complete_Gradient_Estimates}\cite{ART.Wir_Zha.2021.Quantum_OT_Curvature_Dimension_Conditions}. We expect alternatives to be new even in the finite-dimensional setting.\par
We turn to \cite{ART.Wir_Zha.2021.Quantum_OT_Complete_Gradient_Estimates}. First, sufficient conditions for satisfying Bakry-\'Emery conditions as per \cite{ART.Car_Maa.2020.Quantum_OT_III} differ in applicability. Assuming tracial state, \cite{ART.Wir_Zha.2021.Quantum_OT_Complete_Gradient_Estimates} gives sufficient ones as stated above using a novel intertwining property for general families of bounded linear operators. They do not assume direct sum noncommutative gradients and are therefore more general than us in the finite-trace case. This provides means to construct examples for complete gradient estimates stable under tensoring which otherwise appear difficult according to \cite{ART.Wir_Zha.2021.Quantum_OT_Complete_Gradient_Estimates} itself. These do not cover crucial fundamental example classes as per $\CI\rc$, resp.~further iterations on the latter using standard constructions.\par
We assume direct sum noncommutative gradients, rather than complete gradient estimates, but not finite trace. We cover natural examples given by dynamic quantum gradients \cite{ART.Kad.1966.OpAlg_Derivations}, e.g.~intertwining sets of Clifford generators, which indeed have no finite trace. They use tensor product AF-$C^{*}$-bimodules in each summand and thus generalise \cite{ART.Car_Maa.2014.Quantum_OT_I} to infinite dimensions. In the logarithmic mean setting, we use sufficient conditions as per $\BI.5\rc$ to show they have strictly positive lower Ricci bounds. Secondly, functional inequalities in \cite{ART.Wir_Zha.2021.Quantum_OT_Complete_Gradient_Estimates} require use of relative entropy for $W^{*}$-algebras conditioned in the second variable to the given fixed-point subalgebra. Assuming tracial state, \cite{ART.Wir_Zha.2021.Quantum_OT_Complete_Gradient_Estimates} gives functional inequalities as stated above. Recent work of Brannan, Gao and Junge \cite{ART.Bra_Gao_Jun.2022.Quantum_OT_LogSob_I}\cite{ART.Bra_Gao_Jun.2023.Quantum_OT_LogSob_II} independently obtained similar results to Wirth and Zhang \cite{ART.Wir_Zha.2021.Quantum_OT_Complete_Gradient_Estimates} for tracial states using likewise a priori conditioning of relative entropy for $W^{*}$-algebras. Their assumptions imply neither approach covers all fundamental example classes as per $\CI\rc$, in particular our strictly positive case, nor considers its conditioning as determined by the underlying metric geometry. We do show restriction to finitely supported accessibility components is compression of quantum relative entropy with support projections of the given fixed part. We therefore have a conditioning determined by the underlying metric geometry as restriction to finitely supported accessibility components. This is used in the coarse graining process, necessary for our equivalence theorem, and yields non-spatial lower Ricci bounds plus functional inequalities using unconditioned quantum relative entropy as only functional - regardless of finiteness or ergodicity. We thereby ensure functional inequalities reveal properties of the given metric geometry.\par
As explained in the introduction, we study non-spatial lower Ricci bounds as per $\BI.4\rc$ and apply functional inequalities as per $\BI.6\rc$ to probe any underlying metric geometry arising from one of our fundamental example classes as per $\CI\rc$, resp.~an iteration using standard constructions. Following our examination of differences above, we see results in \cite{PRE.Wir.2018.NC_OT}\cite{ART.Wir_Zha.2021.Quantum_OT_Complete_Gradient_Estimates}, as well as \cite{ART.Bra_Gao_Jun.2022.Quantum_OT_LogSob_I}\cite{ART.Bra_Gao_Jun.2023.Quantum_OT_LogSob_II}, are insufficient for our purposes. Conversely, allowing for non-traciality and non-ergodicity lets us cover quantum optimal transport of normal states on arbitrary hyperfinite factors and therefore our motivating application given by first and second quantisation of spectral triples. We appear to have comparatively higher control of fine-structures determined by our initial data, a control we ensure is inherited by all objects we consider through compatibility. We see this bottom-up design yields flexible architecture with the coarse graining process its step-by-step reduction process terminating in a well-behaved finite-dimensional Riemannian setting open to direct calculation. We apply the latter to show stronger results for a wider range of examples in the AF-$C^{*}$-setting covering common algebras of observables in quantum statistical mechanics \cite{BK.Bra.1987.OpAlg_Quantum_StM_I}\cite{BK.Bra.1987.OpAlg_Quantum_StM_II}\cite{BK.Nes_Sto.2006.Rel_Ent}. We view our approach as complementary to \cite{PRE.Wir.2018.NC_OT}.\par 
This concludes our explanation of similarities and differences. We consider the use of our discussion for studying noncommutative gauge theories \cite{ART.Cha_Con.1996.NCG_Spectral_Action_I}\cite{ART.Cha_Con_vSui.2013.NCG_Inner_Fluctuations}\cite{ART.Cha_Con_vSui.2020.NCG_Second_Quantisation}\cite{BK.vSui.2015.NCG_AF_Particle_Physics}\cite{BK.Var.2006.NCG_Elements_Short} within Connes' program of noncommutative geometry \cite{BK.Con.1994.NCG}\cite{COL.Con.2021.NCG_Spectral_POV}\cite{BK.Kha.2013.NCG_Basic}\cite{BK.Kha_Mar.2008.NCG_Invitation}. The program so far lacks a general notion of curvature \cite{COL.Fat_Kha.2019.NCG_Curv_Review}\cite{COL.Les_Mos.2019.NCG_Mod_Curv_Morita_Review} independent of a particular class of spectral triples \cite{ART.Con.1996.NCG_Reconstruction}\cite{COL.Con.2021.NCG_Spectral_POV}\cite{BK.Gra_Var_Fig.2001.NCG_Elements}\cite{BK.Var.2006.NCG_Elements_Short}. Noncommutative tori are a challenge \cite{ART.Con_Mos.2014.NCG_Mod_Curv_Tori_2D}\cite{ART.Don_Gho_Kha.2020.NCG_Tori_Ricci_Curv_3D}\cite{ART.Fat_Kha.2015.NCG_Tori_Scal_Curv_4D}\linebreak\cite{ART.Les_Mos.2016.NCG_Mod_Curv_Morita}. We study the weaker notion of curvature bounds for $\EVI_{\lambda}$-gradient flows driven by l.s.c.~functionals for relevant metric geometries \cite{BK.Amb_Gig_Sav.2008.Classical_OT_GradFlow}\cite{ART.Mur_Sav.2020.Classical_OT_EVI}. The spectral paradigm of noncommutative geometry \cite{ART.Cha_Con.1996.NCG_Spectral_Action_I}\cite{ART.Cha_Con.1997.NCG_Spectral_Action_II}\cite{ART.Cha_Con_Mar.2007.NCG_Standard_Model_Recovered}\cite{ART.Con.1996.NCG_Reconstruction}\cite{COL.Con.2021.NCG_Spectral_POV} based on Gelfand duality \cite{BK.Tak.1979.OpAlg_I} implies a suitable notion must cover continuous, discrete and finally mixed continuous-discrete noncommutative geometries \cite{BK.Gra_Var_Fig.2001.NCG_Elements}\cite{BK.vSui.2015.NCG_AF_Particle_Physics}\cite{BK.Var.2006.NCG_Elements_Short}. Unfortunately, the AF-$C^{*}$-setting does not consider spatial coordinates, i.e.~non-discrete geometries, unless we introduce them in form of pa\-ra\-metrisations for continuous fields of AF-$C^{*}$-algebras \cite{BK.vSui.2015.NCG_AF_Particle_Physics}. First and second quantisation of spectral triples exemplify such lack of spatial coordinates.\par
First quantisation considers commutative spectral triples, i.e.~first quantisation of compact spin manifolds \cite{ART.Con.1996.NCG_Reconstruction}. We show quantum optimal transport is transversal to spatial optimal transport in this case. Second quantisation rectifies this by quantising all spatial coordinates. We apply a characterisation in \cite{ART.Cha_Con_vSui.2020.NCG_Second_Quantisation} to obtain sufficient conditions s.t.~the quantum gradients used are infinitesimal evolution of observables at thermal equilibrium determined by KMS-states \cite{BK.Bra.1987.OpAlg_Quantum_StM_II}. Each assumes fixed gauge field \cite{ART.Cha_Con.1996.NCG_Spectral_Action_I}\cite{BK.vSui.2015.NCG_AF_Particle_Physics}\linebreak\cite{BK.Var.2006.NCG_Elements_Short}. Varying von Neumann entropy \cite{BK.Ohy_Pet.1993.Rel_Ent} of such KMS-states w.r.t.~the canonical trace yields description of the spectral action on gauge fields \cite{ART.Cha_Con.1996.NCG_Spectral_Action_I}\cite{ART.Cha_Con.1997.NCG_Spectral_Action_II}\cite{ART.Cha_Con_Mar.2007.NCG_Standard_Model_Recovered} in terms of quantum statistical mechanics using quantum relative entropy as per $\BI.1\rc$ \cite{ART.Cha_Con_vSui.2020.NCG_Second_Quantisation}. Upon passing to second quantisation, we introduce gauge fields as spatial coordinates. We consider all normalised Radon measures on finite-dimensional spaces of admissible gauge fields evaluating in CAR-algebras \cite{BK.Nes_Sto.2006.Rel_Ent}, i.e.~states on continuous fields of AF-$C^{*}$-algebras. We thereby generalise to quantum optimal transport parametrised by gauge fields and give an internalised spectral action on the aforementioned states using relative entropy for $W^{*}$-algebras. This gives our ansatz as per $\CI\rc$ in Chapter \ref{CH.QOT}. If key technical challenges are solved in future work, then we hope to study the dynamics of such generalised gauge fields described as gradient flows driven by the internalised spectral action for the given parametrised quantum optimal transport. We are motivated by the classical approach of Jordan, Kinderlehrer and Otto for Fokker-Planck equations \cite{ART.Jor_Kin_Ott.1998.Fokker_Planck}\cite{ART.Ott.2001.Classical_OT_Porous_Medium}\cite{ART.Ott.2005.Classical_OT_GradFlow_DisConvex}.\par 
We may relax assumptions on fibres to cover disintegration of tracial $W^{*}$-algebras into direct integrals of hyperfinite factors according to the von Neumann disintegration theorem \cite{BK.Tak.1979.OpAlg_I}. We see fundamental example classes using tracial AF-$C^{*}$-algebras generating hyperfinite factors of type I and II by $\sigma$-weak closure are of particular interest. We thereby define general parametrised quantum optimal transport. We view quantum optimal transport as its pointwise case. We explain states on CAR-algebras are scaling limits of spin states encoding qubits \cite{ART.Bur_DiVi_Loss.1999.QC_Spin_QDots_as_QGates}\cite{ART.Bur_Lad_Nic.2023.QC_Spin_Overview}\cite{BK.Nie_Chu.2000.Quantum_Computation_Information}\cite{ART.DiVi.2000.Criteria}\cite{ART.DiVi_Loss.1998.Quantum_Information_Physical}, but not necessarily pure \cite{ART.Fri_Sio.2011.QC_Spin_Mixed_States}. Using noncommutative conditional expectations \cite{BK.Tak.1979.OpAlg_I}, we therefore consider states on tracial AF-$C^{*}$-algebras as scaling limit of uniformly conditioned spin states encoding a sequence of qubits without use of any spatial coordinates. We view quantum optimal transport as transport of quantum information, and the parametrised one as transport of densities of quantum information over encoding schemes, at the end of Chapter \ref{CH.QOT}.


\chapter{Noncommutative Differential Structures}\label{CH.NCDS}

Noncommutative differential structures collect the data which define quantum optimal transport distances. Each consists of two components and one setting. First, we have an AF-$C^{*}$-bimodule over a, possibly different, tracial AF-$C^{*}$-algebra. This establishes noncommutative topology, measures and integrals. Secondly, we have a quantum gradient for the given AF-$C^{*}$-bimodule. These are noncommutative gradients with likewise chain rule. The relationship between gradients, heat semi\-groups and Dirichlet forms extends to the noncommutative setting \cite{ART.Cip.1997.NC_Dirichlet_Markov}\cite{ART.Cip_Sav.2003.NC_Dirichlet_Grad}. Finally, we have a representing function of an operator mean together with an interpolation factor. This lets us define noncommutative division operators. They determine, and are in turn determined by, quasi-entropies \cite{ART.Hia_Pet.2012.Quasi_Entropy_I}\linebreak\cite{ART.Hia_Pet.2013.Quasi_Entropy_II} used to define energy functionals. In Chapter \ref{CH.QOT}, we readily see our construction of quantum optimal transport distances follows the classical case \cite{ART.Dol_Naz_Sav.2009.Generalised_OT} but using data as above. Thus Banach dual spaces of AF-$C^{*}$-bimodules serve as synthetic tangent spaces for the weak formulation of continuity equations in the AF-$C^{*}$-setting.\par
The data collected is, by definition or construction, compatible with compression and finite-dimensional approximation. These are two general operations we formalise in a coarse graining process. Compatibility transfers to quantum Laplacians, i.e.~Laplacians of quantum gradients, their noncommutative heat semigroups, as well as continuity equations. Compatibility therefore transfers to quantum optimal transport. The coarse graining process formalising the latter is thereby essential for the majority of our results as it reduces the AF-$C^{*}$-setting to the finite-dimensional one s.t.~ergodicity is recovered up to a controlled remainder.

\medskip

\noindent\textbf{Structure.} In Section \ref{SEC.NCDS_AF}, we discuss AF-$C^{*}$-bimodules over tracial AF-$C^{*}$-algebras. In Section \ref{SEC.NCDS_NCD}, we discuss noncommutative division operators. In Section \ref{SEC.NCDS_NCG}, we discuss quantum gradients for AF-$C^{*}$-bimodules. We then define noncommutative differential structures, discuss compatibility and outline the coarse graining process.\par


\newpage



\section[The AF-$C^{*}$-Setting]{The AF-$\mathbf{C}^{*}$-Setting}\label{SEC.NCDS_AF}

AF-$C^{*}$-bimodules over tracial AF-$C^{*}$-algebras are the setting for continuity equations of states compatible with compression and finite-dimensional approximation. Elements in Banach dual spaces of AF-$C^{*}$-bimodules serve as likewise compatible synthetic tangent vectors in our weak formulation. In particular, AF-$C^{*}$-bimodules have an extension of bimodule actions to spaces of measurable operators s.t.~their noncommutative $L^{2}$-spaces are symmetric $W^{*}$-bimodules. The latter are Hilbert spaces on which noncommutative division operators act even upon compression. As such, they provide suitable setting for the Leibniz rule and serve as codomains of quantum gradients.

\medskip

\noindent\textbf{Structure.} In Subsection \ref{SSEC.NCDS_AF_BIM}, we study AF-$C^{*}$-bimodules over tracial AF-$C^{*}$-algebras and extensions of AF-$C^{*}$-bimodule actions. In Subsection \ref{SSEC.NCDS_AF_FC}, we discuss compressed pulled-back joint functional calculus of extended AF-$C^{*}$-bimodule actions.


\subsection[AF-$C^{*}$-bimodules over tracial AF-$C^{*}$-algebras]{AF-$\mathbf{C}^{*}$-bimodules over tracial AF-$\mathbf{C}^{*}$-algebras}\label{SSEC.NCDS_AF_BIM}

AF-$C^{*}$-bimodules over tracial AF-$C^{*}$-algebras are defined using local $^{*}$-homomorphisms of tracial AF-$C^{*}$-algebras. These are $^{*}$-homomorphisms of $C^{*}$-algebras compatible with all AF-$C^{*}$-structures in use, further extending to spaces of measurable operators. Noncommutative $L^{2}$-spaces of AF-$C^{*}$-bimodules are symmetric $W^{*}$-bimodules.


\subsubsection*{Tracial $C^{*}$-algebras and spaces of measurable operators}

Let $(M,\tau)$ be a tracial $W^{*}$-algebra, i.e.~$W^{*}$-algebra $M$ and f.s.n.~trace $\tau:M_{+}\longrightarrow [0,\infty]$ with definition domain $\mathfrak{m}_{\tau}$ \lc{}cf.~Definition \ref{DFN.Wstar_Trace} and Definition \ref{DFN.Wstar_Trace_FSN}\rc{}. Uniform closure of $M$ in measure topology is the space of measurable operators $L^{0}(M,\tau)$ \lc{}cf.~Definition \ref{DFN.Wstar_SMO_I}\rc{}. Algebra involution on $M$ extends to $L^{0}(M,\tau)$. We obtain the space $L^{0}(M,\tau)_{h}$ of self-adjoint, as well as the space $L^{0}(M,\tau)_{+}$ of positive elements \lc{}cf.~Definition \ref{DFN.Wstar_SMO_III}\rc{}. Since $M_{+}$ generates the partial order on $M$ \lc{}cf.~Proposition \ref{PRP.Cstar_PO}\rc{}, note $L^{0}(M,\tau)_{+}$ generates the partial order on $L^{0}(M,\tau)$ by density in measure topology \lc{}cf.~Proposition \ref{PRP.Wstar_NCI_V}\rc{}. We extend the f.s.n.~trace to $\tau:L^{0}(M,\tau)_{+}\longrightarrow [0,\infty]$ \lc{}cf.~Definition \ref{DFN.Wstar_SMO_Trace}\rc{}. For details on $C^{*}$-~and $W^{*}$-algebras, we refer to Subsection \ref{SSEC.A_Fnd_CWstar}. For details on tracial $W^{*}$-algebras and their spaces of measurable operators, we refer to Subsection \ref{SSEC.B_SMO_Wstar_Trace}.\par
Let $p\in [1,\infty]$. Noncommutative $L^{p}$-space $\lc{}L^{p}(M,\tau),\|.\|_{p}\rc\subset L^{0}(M,\tau)$ is a Banach space \lc{}cf.~Definition \ref{DFN.NC_Int_Lp}\rc{}. Algebra involution on $M$ extends to $L^{p}(M,\tau)$. We obtain the space $L^{p}(M,\tau)_{h}$ of self-adjoint, as well as the space $L^{p}(M,\tau)_{+}$ of positive elements. We may decompose accordingly \lc{}cf.~Proposition \ref{PRP.Wstar_NCI_IV}\rc{}. If $p=1$, then $\tau\in L^{1}(A,\tau)_{+}^{*}$ \lc{}cf.~$3)$ in Proposition \ref{PRP.Wstar_NCI_I}\rc{}. If $p=2$, then $\lc{}L^{2}(M,\tau),\|.\|_{2}\rc$ is a Hilbert space. If $p=\infty$, then $\lc{}L^{\infty}(M,\tau),\|.\|_{\infty}\rc{}=\lc{}M,\|.\|_{M}\rc$. Noncommutative $L^{p}$-spaces fulfil H\"older inequalities. Note Definition \ref{DFN.Wstar_Trace_MSP_Musical} uses the modified standard pairing as per Remark \ref{REM.Wstar_Trace_MSP}. For details on noncommutative integration, we refer to Subsection \ref{SSEC.B_SMO_NCI}.

\begin{dfn}\label{DFN.Wstar_Trace_MSP_Musical}
For all $\mu\in L^{1}(M,\tau)^{\flat}$, let $\sharp\mu\in L^{1}(M,\tau)$ be unique s.t.~$\mu=\lc\sharp\mu\rc^{\flat}$. If $p=q=2$, then set $\sharp:=\flat^{-1}\in\BII\lc{}L^{2}(M,\tau)\rc$ and call $\lc\flat,\sharp\rc$ musical isomorphisms on $L^{2}(M,\tau)$.
\end{dfn}

\begin{rem}\label{REM.Wstar_Trace_MSP}
Let $p,q\in [1,\infty]$. If $1=p^{-1}+q^{-1}$, then the modified standard pairing 

\begin{align}\label{EQ.REM.Wstar_Trace_MSP_1}
(x,y)\mapsto x^{\flat}(y)=\tau(x^{*}y)
\end{align}

\noindent defined on $L^{p}(A,\tau)\times L^{q}(A,\tau)$ is bounded, anti-linear in the first and linear in the second variable, as well as non-degenerate \lc{}cf.~Definition \ref{DFN.Wstar_NCI_MSP} and Proposition \ref{PRP.Wstar_NCI_VI}\rc{}. For all $x\in L^{p}(M,\tau)$ and $y\in L^{q}(A,\tau)$, get $\tau(x^{*}y)=\tau\lc{}yx^{*}\rc$ and $\overline{\tau(x^{*}y)}=\tau\lc{}xy^{*}\rc$ by traciality.\par
If $p=1$ and $q=\infty$, then $\flat:L^{1}(M,\tau)\longrightarrow M^{*}$ is positivity-preserving and anti-linear isometry onto the set $M_{*}\subset M^{*}$ of all normal bounded functional on $M$ equipped with the dual space partial order \lc{}cf.~Proposition \ref{PRP.Wstar_NCI_VI} and Remark \ref{REM.Wstar_NCI_MSP}\rc{}. If $A\subset M$ is a $\sigma$-weakly dense $C^{*}$-subalgebra, then $A\subset M$ is strongly dense. Normality therefore yields $L^{1}(A,\tau)^{\flat}\subset A^{*}$ as partially ordered Banach spaces. For all $\mu\in L^{1}(A,\tau)^{\flat}$, get unique $\sharp\mu\in L^{1}(A,\tau)$ s.t.~$\mu=\lc\sharp\mu\rc^{\flat}$. If $p=q=2$, then $\flat\in\GL\lc\BII\lc{}L^{2}(M,\tau)\rc\rc$.
\end{rem}

Positive elements generate the partial order on $C^{*}$-~and $W^{*}$-algebras, as well as their Banach dual spaces \lc{}cf.~Definition \ref{DFN.Cstar} and Proposition \ref{PRP.Cstar_PO}\rc{}. Definition \ref{DFN.Cstar_Trace_Abstract} gives abstract tracial $C^{*}$-algebras \lc{}cf.~Remark \ref{REM.Abstract_Concrete}\rc{}. Following Remark \ref{REM.Cstar_Trace_Abstract_Concrete}, the latter extends Definition \ref{DFN.Wstar_Trace} and subsumes the concrete case s.t.~we have consistent use of canonical left-~and right-actions for joint functional calculus of self-adjoint measurable operators. As consequence, compressing with projections as per Lemma \ref{LEM.Cstar_Trace_Abstract_Projection} extends readily from one to two variables as special case of the tracial $W^{*}$-algebra setting.

\begin{dfn}\label{DFN.Cstar_Trace_Abstract}
Let $A\subset M$ be a $\sigma$-weakly dense $C^{*}$-subalgebra. We call $(A,\tau)$ a tracial $C^{*}$-algebra in $M$. Set $1_{A}:=1_{M}$.

\begin{itemize}
\item[1)] Set $L^{0}(A,\tau):=L^{0}(M,\tau)$ and

\begin{align}\label{EQ.DFN.Cstar_Trace_Abstract_1}
L^{0}(A,\tau)_{h}:=L^{0}(M,\tau)_{h},\ L^{0}(A,\tau)_{+}:=L^{0}(M,\tau)_{+}.    
\end{align}

\begin{reapply}
\end{reapply}

\item[2)] For all $p\in [1,\infty]$, set $\lc{}L^{p}(A,\tau),\|.\|_{p}\rc{}:=\lc{}L^{p}(M,\tau),\|.\|_{p}\rc$ and

\begin{align}\label{EQ.DFN.Cstar_Trace_Abstract_2}
L^{p}(A,\tau)_{h}:=L^{p}(M,\tau)_{h},\ L^{p}(A,\tau)_{+}:=L^{p}(M,\tau)_{+}.
\end{align}

\begin{reapply}
\end{reapply}

\item[3)] For all $p,q\in [1,\infty]$, set $L^{p,q}(A,\tau):=L^{p}(A,\tau)\cap L^{q}(A,\tau)$ and

\begin{align}\label{EQ.DFN.Cstar_Trace_Abstract_3}
L^{p,q}(A,\tau)_{h}:=L^{p}(A,\tau)_{h}\cap L^{q}(A,\tau)_{h},\ L^{p,q}(A,\tau)_{+}:=L^{p}(A,\tau)_{+}\cap L^{q}(A,\tau)_{+}.
\end{align}

\begin{reapply}
\end{reapply}

\end{itemize}
\end{dfn}

\begin{ntn}\label{NTN.Cstar_Trace_Abstract}
Unless stated otherwise, we write $(A,\tau)$ and $\|.\|_{\tau}=\|.\|_{2}$ for all $\sigma$-weakly dense $C^{*}$-subalgebras $A\subset M$. This differs from the distinct notation $\lc\HII(M,\tau),\|.\|_{\tau}\rc$ and $\lc{}L^{2}(M,\tau),\|.\|_{2}\rc$ used in the appendix. Notation remains unambiguous throughout since we only use $\lc{}L^{2}(A,\tau),\|.\|_{\tau}\rc{}=\lc{}L^{2}(M,\tau),\|.\|_{2}\rc$ in the main matter.
\end{ntn}

\begin{rem}\label{REM.Cstar_Trace_Abstract_Concrete}
Let $(A,\tau)$ be a tracial $C^{*}$-algebra and $\LII:A\longrightarrow\BII\lc\HII(A,\tau)\rc$ canonical left-action of $A$ on $\HII(A,\tau)$ \lc{}cf.~Definition \ref{DFN.Wstar_Trace} and Definition \ref{DFN.Wstar_Trace_CLA}\rc{}. Using normal extension \lc{}cf.~Proposition \ref{PRP.Wstar_Equivalence} and Proposition \ref{PRP.Wstar_Trace_Ext_I}\rc{}, get tracial $W^{*}$-algebra $\lc\mathcal{L}(A)'',\tau\rc$ with $\sigma$-weakly dense $C^{*}$-subalgebra $\LII(A)\subset\LII(A)''$. We thereby construct the tracial $C^{*}$-algebra $\lc\mathcal{L}(A),\tau\rc$ in $\mathcal{L}(A)''$. This is the concrete case.\par
If $(A,\tau)$ is a tracial $C^{*}$-algebra in $M$, then the canonical left-action $\LII$ of $A$ on $\HII(A,\tau)$ is not the canonical left-action $L$ of $M$ on $L^{2}(M,\tau)$ in general \lc{}cf.~Definition \ref{DFN.Wstar_CLRA}\rc{}. Note $L$ subsumes $\LII$ by twisting with natural Hilbert space isometry $\HII(A,\tau)\cong L^{2}(M,\tau)$ \lc{}cf.~Proposition \ref{PRP.Wstar_CLRA_I}\rc{}, $L$ is given by multiplication in $L^{0}(M,\tau)$, as well as inclusions $A\subset M\subset L^{0}(M,\tau)$ of $^{*}$-subalgebras. The analogous holds for canonical right-actions and opposite algebras. Altogether, requiring $M$ in Definition \ref{DFN.Cstar_Trace_Abstract} avoids difficulties arising from identification of $A\cong\LII(A)$, as is common yet implicit in the literature, while using canonical left-~and right-actions for joint functional calculus of self-adjoint measurable operators \lc{}cf.~Remark \ref{REM.Wstar_CLRA}\rc{}
\end{rem}

Note suitable inclusion maps of Banach dual spaces arise from Banach duals of noncommutative conditional expectations. For abstract tracial $C^{*}$-algebras, Definition \ref{DFN.Cstar_Trace_Abstract_Dualisation} gives inclusion maps obtained from compressing with projections in $W^{*}$-algebras. This uses abstract compression maps. Assuming positivity and fixed norm, get injectivity in the non-unital case as per $1)$ in Proposition \ref{PRP.Cstar_Trace_Abstract_Dualisation_II}. For tracial AF-$C^{*}$-algebras, further note Definition \ref{DFN.AF_Cstar_Trace_Dualisation} gives inclusion maps obtained from Hilbert space projections to generating $C^{*}$-subalgebras. This additionally yields restriction maps.\par
We compress $C^{*}$-subalgebras with projections. Let $A\subset M$ be a $C^{*}$-subalgebra and $p\in M$ be a projection. We have compressed $C^{*}$-subalgebra $A[p]=pC^{*}\lc{}A,p\rc{}p\subset M$ \lc{}cf.~$2)$ in Definition \ref{DFN.Compression_Abstract_Bd}\rc{}. If $p=1_{M}$, then we recover the unitalisation $A[1_{M}]=C^{*}\lc{}A,1_{M}\rc$ of $A$ in $M$ \lc{}cf.~Definition \ref{DFN.Cstar_Unitalisation}\rc{}. If $A=M$, then $M[p]=pMp\subset (M,\tau)$ is a semi-finite $W^{*}$-subalgebra \lc{}cf.~Remark \ref{REM.Compression_Abstract_Bd}, Definition \ref{DFN.Wstar_Trace_SF_Subalg} and $2)$ in Proposition \ref{PRP.Wstar_Trace_NCE_II}\rc{}. We have tracial $W^{*}$-algebra $\lc{}M[p],\tau\rc$ \lc{}cf.~$1)$ in Proposition \ref{PRP.Wstar_Trace_SF_Subalg}\rc{}. Note $p^{\perp}=1_{M}-p$ and $M[p][1_{M}]=M[p]\oplus\langle p^{\perp}\rangle_{\mathbb{C}}$ \lc{}cf.~Proposition \ref{PRP.Wstar_Unitalisation}\rc{}. Assume $A\subset M$ is a $\sigma$-weakly dense $C^{*}$-subalgebra. The compressed $C^{*}$-subalgebra $A[p]\subset M[p]$ is $\sigma$-weakly dense itself in this case. We moreover have $A[p][1_{M}]=A[p]\oplus\langle p^{\perp}\rangle_{\mathbb{C}}$ \lc{}cf.~Proposition \ref{PRP.Cstar_Unitalisation}\rc{}.

\begin{lem}\label{LEM.Cstar_Trace_Abstract_Projection}
Let $A\subset M$ be a $\sigma$-weakly dense $C^{*}$-subalgebra. For all projections $p\in M$ and $q\in [1,\infty]$, we have

\vspace{-0.02787125cm}
\begin{itemize}
\item[1)] tracial $C^{*}$-algebra $(A[p],\tau)$ in $M[p]$,

\vspace{-0.02787125cm}
\item[2)] $L^{0}(A[p],\tau)=pL^{0}(A,\tau)p$ and $L^{q}(A[p],\tau)=pL^{q}(A,\tau)p$.
\end{itemize}
\end{lem}
\begin{proof}
We know $1)$. Thus $L^{\infty}(A[p],\tau)=M[p]$, hence $2)$ follows by Proposition \ref{PRP.Wstar_L2Red_III}.
\end{proof}

The abstract compression map $\comp:A[1_{M}]\longrightarrow A[p]$ is given by $\comp x=pxp$ for all $x\in A[1_{M}]$ \lc{}cf.~Definition \ref{DFN.Compression_Abstract_Bd}\rc{}. Note $\comp$ is a completely positive, normal, unital and surjective bounded linear map \lc{}cf.~Proposition \ref{PRP.Compression_Abstract_Bd}\rc{}. If $A=M$, then we recover noncommutative conditional expectations as per Remark \ref{REM.Cstar_Trace_Abstract_Dualisation}. For details on compressed $C^{*}$-subalgebras and their abstract compression maps, we refer to Subsection \ref{SSEC.A_Maps_Compression}. For details on semi-finite $W^{*}$-subalgebras, we refer to Subsection \ref{SSEC.B_JFC_L2Red}.\par


\pagebreak


We have positivity-preserving injective Banach dual $\com_{p}^{*}:A[p]^{*}\longrightarrow A[1_{M}]^{*}$. If we restrict to $A\subset A[1_{M}]$, then we further have positivity-preserving bounded linear map $\com_{p}^{*}:A[p]^{*}\longrightarrow A^{*}$. The latter is not injective in general. If $q\in L^{\infty}(A,\tau)$ is a projection s.t.~$p\leq q$, then $pq=p$ implies $\comp\lc{}A\lb{}q\rb\rc{}=A[p]$. Get positivity-preserving injective Banach dual $\com_{p}^{*}:A[p]^{*}\longrightarrow A\lb{}q\rb^{*}\subset A[1_{M}]^{*}$ in this case.

\begin{dfn}\label{DFN.Cstar_Trace_Abstract_Dualisation}
Let $A\subset M$ be a $C^{*}$-subalgebra. For all projections $p\in L^{\infty}(A,\tau)$, we define the $p$-th inclusion $\incp:=\com_{p}^{*}:A[p]^{*}\longrightarrow A[1_{M}]^{*}$.
\end{dfn}

\begin{rem}\label{REM.Cstar_Trace_Abstract_Dualisation}
Semi-finite $W^{*}$-subalgebras have unique noncommutative conditional expectations \lc{}cf.~Definition \ref{DFN.Wstar_Trace_NCE}, Remark \ref{REM.Wstar_Trace_NCE} and Definition \ref{DFN.Wstar_Trace_NCE_MSP}\rc{}. For all projections $p\in M$, $\comp:M\longrightarrow M[p]$ is the noncommutative conditional expectation $\pi_{M[p]}^{M}$ from $M$ to $M[p]$ \lc{}cf.~Proposition \ref{PRP.Wstar_Trace_NCE_I} and $2)$ in Proposition \ref{PRP.Wstar_Trace_NCE_II}\rc{}.
\end{rem}

\begin{prp}\label{PRP.Cstar_Trace_Abstract_Dualisation_I}
Let $A\subset M$ be a $\sigma$-weakly dense $C^{*}$-subalgebra. All inclusion maps in Definition \ref{DFN.Cstar_Trace_Abstract_Dualisation} are bounded linear, positivity-preserving and injective. They furthermore satisfy the following.

\begin{itemize}
\item[1)] All inclusion maps in Definition \ref{DFN.Cstar_Trace_Abstract_Dualisation} are $w^{*}$-continuous.

\item[2)] For all projections $p\leq q$ in $L^{\infty}(A,\tau)$, we have $A[p]^{*}\subset A\lb{}q\rb^{*}\subset A[1_{A}]^{*}$ as partially ordered Banach spaces. 
\end{itemize}
\end{prp}
\begin{proof}
Bounded linearity and $1)$ are immediate. Let $p\in L^{\infty}(A,\tau)$ be a projection. Since $\comp$ is positivity-preserving \lc{}cf.~Proposition \ref{PRP.Compression_Abstract_Bd}\rc{}, $\incp$ is as well. If $q\in L^{\infty}(A,\tau)$ is a projection s.t.~$p\leq q$, then $pq=p$ implies $\comp\circ\hspace{0.0275cm} \com_{q}=\comp$ and therefore $2)$.
\end{proof}

\begin{ntn}\label{NTN.Cstar_Trace_Abstract_Dualisation}
Let $A\subset M$ be a $\sigma$-weakly dense $C^{*}$-subalgebra. For all projections $p\leq q$ in $L^{\infty}(A,\tau)$, we suppress $\incp$ and $\incq$ on $A[p]^{*}$.
\end{ntn}

States on abstract tracial $C^{*}$-algebras are noncommutative probability measures. They are normal if they have noncommutative density. Equation \ref{EQ.DFN.Cstar_Trace_Abstract_State_Space_1} and Equation \ref{EQ.DFN.Cstar_Trace_Abstract_State_Space_2} use spectral measures and spectra, as well as bounded measurable functional calculus of self-adjoint measurable operators \lc{}cf.~Definition \ref{DFN.Wstar_CLRA_FC_I} and Lemma \ref{LEM.Wstar_CLRA_FC}\rc{}.

\begin{dfn}\label{DFN.Cstar_Trace_Abstract_State_Space}
Let $A\subset M$ be a $\sigma$-weakly dense $C^{*}$-subalgebra.

\begin{itemize}
\item[1)] We define the state space $\SII(A):=\lset\mu\in A_{+}^{*}\ \vset\ \|\mu\|_{A}=1\rset$ and the normal state space $\mathcal{S}^{\NI}(A):=\SII(A)\cap L^{1}(A,\tau)^{\flat}$ of $A$. Set

\begin{align}\label{EQ.DFN.Cstar_Trace_Abstract_State_Space_1}
\mathcal{S}_{>0}^{\NI}(A):=\mathcal{S}^{\NI}(A)\cap\lset{}x\in L^{1}(A,\tau)_{h}\ \vset\ \Gamma_{x,L^{\infty}(A,\tau)}(\delta_{0})=0\rset^{\flat}.
\end{align}

\begin{reapply}
\end{reapply}

\item[2)] For all $p\in (1,\infty]$, set $\mathcal{S}^{\NI,p}(A):=\SII(A)\cap L^{1,p}(A,\tau)^{\flat}$ and

\begin{align}\label{EQ.DFN.Cstar_Trace_Abstract_State_Space_2}
\mathcal{S}_{-1}^{\NI,p}(A):=\mathcal{S}^{\NI,p}(A)\cap\lset{}x\in L^{1}(A,\tau)_{h}\ \vset\ 0\notin\specM x\rset^{\flat}.  
\end{align}

\begin{reapply}
\end{reapply}

\end{itemize}
\end{dfn}

\begin{rem}\label{REM.Cstar_Trace_Abstract_State_Space}
For all $\sigma$-weakly dense $C^{*}$-subalgebras $A\subset M$, our construction and Remark \ref{REM.Wstar_Trace_MSP} shows $\mathcal{S}^{\NI}(A)=\SII(M)\cap L^{1}(M,\tau)^{\flat}=\mathcal{S}^{\NI}(M)$ \lc{}cf.~Definition \ref{DFN.Wstar_NCI_State_Space}\rc{}. For all $\sigma$-weakly dense $C^{*}$-subalgebras $A\subset M$, projections $p\in L^{\infty}(A,\tau)$ and $x\in L^{1}(A[p],\tau)$, we have $\lc\incp x^{\flat}\rc{}(y)=\tau(x^{*}y)$ for all $y\in A[p]$. Lemma \ref{LEM.Cstar_Trace_Abstract_Projection} shows $x=xp=px$ in each case. 
\end{rem}

\begin{prp}\label{PRP.Cstar_Trace_Abstract_Dualisation_II}
Let $A\subset M$ be a $\sigma$-weakly dense $C^{*}$-subalgebra. For all projections $p\leq q$ in $L^{\infty}(A,\tau)$ and $r\in (1,\infty]$, we have

\begin{itemize}
\item[1)] $\SII(A[p])\subset\SII\lc{}A\lb{}q\rb\rc\subset\SII(A)$ and $\mathcal{S}^{\NI}(A[p])\subset\mathcal{S}^{\NI}\lc{}A\lb{}q\rb\rc\subset\mathcal{S}^{\NI}(A)$,

\item[2)] $\mathcal{S}^{\NI,r}(A[p])\subset\mathcal{S}^{\NI,r}\lc{}A\lb{}q\rb\rc\subset\mathcal{S}^{\NI,r}(A)$.
\end{itemize}
\end{prp}
\begin{proof}
Proposition \ref{PRP.Cstar_Trace_Abstract_Dualisation_I} shows positive elements are preserved. Let $p$ in $L^{\infty}(A,\tau)$ be a projection. Since $\mu(1_{A})=\mu\lc{}p\rc{}=\|\mu\|_{A[p]^{*}}$ for all $\mu\in A[p]_{+}^{*}\subset A[1_{A}]^{*}$, fixed norm ensures injectivity upon restriction to $A\subset A[1_{A}]$. Get $1)$. Using the latter, get $2)$.
\end{proof}


\subsubsection*{Tracial AF-$C^{*}$-algebras}

Approximately finite-dimensional, or AF-$C^{*}$-algebras are all $C^{*}$-algebras which are norm closures of ascending chains of finite-dimensional $C^{*}$-algebras. We index all AF-$C^{*}$-algebras over $\mathbb{N}$. This is equivalent to using countable directed sets by existence of cofinal subsets isomorphic to $\mathbb{N}$. Tracial AF-$C^{*}$-algebras are both AF-$C^{*}$-algebras and abstract tracial $C^{*}$-algebras.

\begin{dfn}\label{DFN.AF_Cstar}
Let $A$ be a $C^{*}$-algebra.

\begin{itemize}
\item[1)] A sequence $\lset{}A_{j}\rset_{j\in\mathbb{N}}$ of finite-dimensional $C^{*}$-algebras is ascending if $A_{j}\subset A_{j+1}$ is a $C^{*}$-subalgebra for all $j\in\mathbb{N}$. Set $A_{0}:=\bigcup_{j\in\mathbb{N}}A_{j}$.

\item[2)] We call $A$ an AF-$C^{*}$-algebra if $A=\overline{A_{0}}^{\|.\|_{A}}$ for an ascending sequence $\lset{}A_{j}\rset_{j\in\mathbb{N}}$ of finite-dimensional $C^{*}$-subalgebras. We further call $\lset{}A_{j}\rset_{j\in\mathbb{N}}$ a generating sequence of $A$ and say that $A$ is generated by $\lset{}A_{j}\rset_{j\in\mathbb{N}}$.

\item[3)] If $A$ is an AF-$C^{*}$-algebra generated by $\lset{}A_{j}\rset_{j\in\mathbb{N}}$, then we say that $A$ is

\begin{itemize}
\item[3.1)] strongly unital if $1_{A_{j}}=1_{A_{k}}$ for all $j,k\in\mathbb{N}$,

\item[3.2)] finite-dimensional if $\dim_{\mathbb{C}}A<\infty$,

\item[3.3)] finite if $A=A_{j}$ for all $j\in\mathbb{N}$.
\end{itemize}

\begin{reapply}
\end{reapply}

\end{itemize}
\end{dfn}

\begin{ntn}\label{NTN.AF_Cstar_Fin_Isometry}
For all $n\in\mathbb{N}$, $I_{n}\in M_{n}(\mathbb{C})$ denotes the unit and $\tr_{n}$ the non-normalised canonical trace. In infinite dimensions, i.e.~$n=\infty$, we suppress the subscript and write $I$ and $\tr$. Up to $C^{*}$-isometries, finite-dimensional $C^{*}$-algebras are of form $\oplus_{l=1}^{n}M_{n_{l}}(\mathbb{C})$ for $n\in\mathbb{N}$ \cite{BK.Bro_Oza.2008.Cstar_AF}. If $A$ is a AF-$C^{*}$-algebra generated by $\lset{}A_{j}\rset_{j\in\mathbb{N}}$, then we fix $C^{*}$-isometries

\begin{align}\label{EQ.NTN.AF_Cstar_Fin_Isometry_1}
r_{A}:=\lset r_{A_{j}}:A_{j}\longrightarrow\oplus_{l=1}^{n_{j}}M_{n_{j,l}}(\mathbb{C})\rset_{j\in\mathbb{N}}.
\end{align}

\noindent If $A$ is furthermore finite, then set $\lset{}r_{A_{j}}\rset_{j\in\mathbb{N}}$ to be constant unless stated otherwise.
\end{ntn}

\begin{prp}\label{PRP.AF_Cstar_Unit}
Let $A$ be an AF-$C^{*}$-algebra and $M$ a $W^{*}$-algebra. If $A$ is generated by $\lset{}A_{j}\rset_{j\in\mathbb{N}}$, then 

\begin{itemize}
\item[1)] $\lset{}1_{A_{j}}\rset_{j\in\mathbb{N}}\subset A$ is a left-~and right-approximate identity in $A$,

\item[2)] $1_{M}=\s$-$\lim_{j\in\mathbb{N}}1_{A_{j}}$ if $A\subset M$ is a $\sigma$-weakly dense $C^{*}$-subalgebra.
\end{itemize}
\end{prp}
\begin{proof}
Get $1)$ since $\bigcup_{j\in\mathbb{N}}A_{j}\subset A$ is $\|.\|_{A}$-dense and $1_{A_{j}}\lc{}1_{A_{k}}-1_{A_{j}}\rc{}=\lc{}1_{A_{k}}-1_{A_{j}}\rc{}1_{A_{j}}=0$ for all $j\leq k$ in $\mathbb{N}$. Get $2)$ by $1)$ and uniqueness of units in $C^{*}$-algebras.
\end{proof}

\begin{rem}
Note $1)$ in Proposition \ref{PRP.AF_Cstar_Unit} shows strong unitality implies unitality. In the setting of $2)$ in Proposition \ref{PRP.AF_Cstar_Unit}, we have $1_{M}=1_{A}$ if $A$ is unital.
\end{rem}

Definition \ref{DFN.AF_Cstar_Trace_Abstract} gives tracial AF-$C^{*}$-algebras using abstract formulation. Following Remark \ref{REM.Cstar_Trace_Abstract_Concrete} and Remark \ref{REM.AF_Cstar_Trace_Abstract_Concrete}, the latter therefore subsumes the concrete case in the AF-$C^{*}$-setting s.t.~we have consistent use of canonical left-~and right-actions for joint functional calculus of self-adjoint measurable operators.

\begin{dfn}\label{DFN.AF_Cstar_Trace_Abstract}
Let $A$ be an AF-$C^{*}$-algebra generated by $\lset{}A_{j}\rset_{j\in\mathbb{N}}$ and $(M,\tau)$ a tracial $W^{*}$-algebra. We call $(A,\tau)$ a tracial AF-$C^{*}$-algebra in $M$ generated by $\lset{}A_{j}\rset_{j\in\mathbb{N}}$ if $A\subset M$ is a $\sigma$-weakly dense $C^{*}$-subalgebra and $A_{0}\subset\mathfrak{m}_{\tau}$.
\end{dfn}

\begin{rem}\label{REM.AF_Cstar_Trace_Abstract_Concrete}
Let $(A,\tau)$ be a tracial $C^{*}$-algebra and AF-$C^{*}$-algebra generated by $\lset{}A_{j}\rset_{j\in\mathbb{N}}$ s.t.~$A_{0}\subset\mathfrak{m}_{\tau}$. Following construction in Remark \ref{REM.Cstar_Trace_Abstract_Concrete}, get tracial AF-$C^{*}$-algebra $\lc\LII(A),\tau\rc$ in $\LII(A)''$ generated by $\lset\hspace{-0.0325cm} \LII(A)_{j}\rset_{j\in\mathbb{N}}:=\lset\hspace{-0.0325cm} \LII(A_{j})\rset_{j\in\mathbb{N}}$. This is the concrete case of the AF-$C^{*}$-setting. Note this requires $\LII$ to be a faithful $^{*}$-representation.
\end{rem}

\begin{prp}\label{PRP.AF_Cstar_Trace_I}
For all tracial AF-$C^{*}$-algebras $(A,\tau)$, we have

\begin{itemize}
\item[1)] $A_{0}\subset L^{\infty}(A,\tau)$ is strongly dense,

\item[2)] $A_{0}\subset L^{2}(A,\tau)$ is $\|.\|_{\tau}$-dense,

\item[3)] $A_{0}\subset L^{1}(A,\tau)$ is $\|.\|_{1}$-dense.
\end{itemize}
\end{prp}
\begin{proof}
We show Proposition \ref{PRP.Wstar_NCI_VIII} applies. We know $A_{0}\subset A$ is $\|.\|_{A}$-dense. We show $A_{0}\subset L^{2}(A,\tau)$ is $\|.\|_{\tau}$-dense. For all $j\in\mathbb{N}$, set $\mathcal{A}_{j}:=1_{A_{j}}A_{0}1_{A_{j}}\subset A_{0}$ and note

\begin{align}\label{EQ.PRP.AF_Cstar_Trace_I_1}
\mathcal{M}_{j}:=\overline{\mathcal{A}_{j}}=L^{\infty}(A,\tau)[1_{A_{j}}]=1_{A_{j}}L^{\infty}(A,\tau)1_{A_{j}}\subset L^{\infty}(A,\tau)
\end{align}

\noindent w.r.t.~closure in strong operator topology \lc{}cf.~$2)$ in Definition \ref{DFN.Compression_Abstract_Bd}\rc{}. Equation \ref{EQ.PRP.AF_Cstar_Trace_I_1} shows $L^{2}(\mathcal{M}_{j},\tau)=1_{A_{j}}L^{2}(A,\tau)1_{A_{j}}$ by Lemma \ref{LEM.Cstar_Trace_Abstract_Projection} in each case. Thus $\bigcup_{j\in\mathbb{N}}L^{2}(\mathcal{M}_{j},\tau)\subset L^{2}(A,\tau)$ is $\|.\|_{\tau}$-dense by Proposition \ref{PRP.AF_Cstar_Unit}, hence

\begin{align}\label{EQ.PRP.AF_Cstar_Trace_I_2}
L^{2}(A,\tau)=\overline{\bigcup_{j\in\mathbb{N}}L^{2}(\mathcal{M}_{j},\tau)}^{\|.\|_{\tau}}\subset\overline{\bigcup_{j\in\mathbb{N}}\mathcal{A}_{j}}^{\|.\|_{\tau}}\subset\overline{A_{0}}^{\|.\|_{\tau}}\subset L^{2}(A,\tau).
\end{align}

\noindent Equation \ref{EQ.PRP.AF_Cstar_Trace_I_2} shows $A_{0}\subset L^{2}(A,\tau)$ is $\|.\|_{\tau}$-dense.
\end{proof}

We consider the finite-dimensional setting. Example \ref{BSP.AF_Cstar_Trace_Fin} provides finite case. We discuss the case of restricting to generating $C^{*}$-subalgebras. We use Notation \ref{NTN.PO}. For all AF-$C^{*}$-algebras $A$ generated by $\lset{}A_{j}\rset_{j\in\mathbb{N}}$ and $j\in\mathbb{N}$, $A_{j,h}$ denotes the self-adjoint and $A_{j,+}$ the positive elements in $A_{j}$.

\begin{bsp}\label{BSP.AF_Cstar_Trace_Fin}
Let $(A,\tau)$ be a finite-dimensional tracial $C^{*}$-algebra. Note $L^{\infty}(A,\tau)=A$ by $\sigma$-weak density. For all $j\in\mathbb{N}$, set $A_{j}=A$. This defines finite tracial AF-$C^{*}$-algebra $(A,\tau)$ in $A$. Finiteness does not hold in general.
\end{bsp}

\begin{dfn}\label{DFN.AF_Cstar_Trace_Restriction}
Let $(A,\tau)$ be a tracial AF-$C^{*}$-algebra. For all $j\in\mathbb{N}$, set $\tau_{j}:=\tau\vert_{A_{j}}$ and we define sequence of finite-dimensional $C^{*}$-algebras by setting

\begin{align*}
A_{j,l}:=
\begin{cases}
A_{l} & \If\ l<j\ \textrm{in}\ \mathbb{N}, \\
A_{j} & \Else.
\end{cases}
\end{align*}
\end{dfn}

\begin{rem}\label{REM.AF_Cstar_Trace_Fin}
Unless stated otherwise, any finite-dimensional tracial $C^{*}$-algebra we consider alone is a finite tracial AF-$C^{*}$-algebras as per Example \ref{BSP.AF_Cstar_Trace_Fin}. For generating $C^{*}$-subalgebras, we instead restrict as per Definition \ref{DFN.AF_Cstar_Trace_Restriction}.
\end{rem}

\begin{prp}\label{PRP.AF_Cstar_Trace_II}
Let $(A,\tau)$ be a tracial AF-$C^{*}$-algebra. For all $j\in\mathbb{N}$, we have

\begin{itemize}
\item[1)] tracial AF-$C^{*}$-algebra $(A_{j},\tau)=(A_{j},\tau_{j})$ in $A_{j}$ generated by $\lset{}A_{j,l}\rset_{l\in\mathbb{N}}$,

\item[2)] $\tau_{j}=\oplus_{l=1}^{n_{j}}C_{j,l}\tr_{n_{j,l}}\circ~ r_{A_{j}}$ with $C_{j,l}>0$ for all $l\in\lset{}1,\ldots,n_{j}\rset$,

\item[3)] $r_{A_{j}}(1_{A_{j}})=\sum_{l=1}^{n_{j}}I_{n_{j,l}}$.
\end{itemize}
\end{prp}
\begin{proof}
We have $1)$ since $A_{0}\subset\mathfrak{m}_{\tau}$. Restricting to summands shows $2)$ by uniqueness of the normalised trace on full matrix algebras. Get $3)$ by unitality.
\end{proof}

\begin{rem}\label{REM.AF_Cstar_Trace_Restriction}
For all $j\in\mathbb{N}$, $\lgl\hspace{0.025cm}.\hspace{0.025cm},\hspace{0.025cm}.\hspace{0.025cm}\rgl_{\tau\vert_{A_{j}}}$ equals $\sum_{l=1}^{n_{j}}C_{j,l}\lgl\hspace{0.025cm}.\hspace{0.025cm},\hspace{0.025cm}.\hspace{0.025cm}\rgl_{\tr_{n_{j,l}}}$ pulled back along $r_{A_{j}}^{-1}$.
\end{rem}

\begin{prp}\label{PRP.AF_Cstar_Trace_III}
Let $(A,\tau)$ be a tracial AF-$C^{*}$-algebra. For all $j\in\mathbb{N}$, we consider the Hilbert space projection $\pi_{j}^{A}:L^{2}(A,\tau)\longrightarrow A_{j}$. We have

\begin{itemize}
\item[1)] $\bigcup_{j\in\mathbb{N}}A_{j,+}\subset A_{+}$ is $\|.\|_{A}$-dense,

\item[2)] $\bigcup_{j\in\mathbb{N}}A_{j,+}\subset L^{\infty}(A,\tau)_{+}$ is strongly dense,

\item[3)] $I_{L^{2}(A,\tau)}=\s$-$\lim_{j\in\mathbb{N}}\pi_{j}^{A}$.
\end{itemize}
\end{prp}
\begin{proof}
For all $j\in\mathbb{N}$, get $A_{j,+}\subset A_{+}\subset L^{\infty}(A,\tau)$. If $\{x_{n}\}_{n\in\mathbb{N}}\subset A_{0}$ s.t.~$\|.\|_{A}$-$\lim_{n\in\mathbb{N}}x_{n}=x\geq 0$ in $A$, then $\|.\|_{A}$-$\lim_{n\in\mathbb{N}}\max\{x_{n},0\}=x$. Using strong convergence, the analogous statement follows if $x\geq 0$ in $L^{\infty}(A,\tau)$. This shows $1)$ and $2)$. We know $A_{0}\subset L^{2}(A,\tau)$ is $\|.\|_{\tau}$-dense by Proposition \ref{PRP.AF_Cstar_Trace_I}. Thus $\|.\|_{\tau}$-$\lim_{j\in\mathbb{N}}\pi_{j}^{A}(x)=x$ for all $x\in A_{0}$, hence $3)$ follows.
\end{proof}
	

\subsubsection*{Banach dual spaces of tracial AF-$C^{*}$-algebras}

Inclusion and restriction maps of tracial AF-$C^{*}$-algebras in Definition \ref{DFN.AF_Cstar_Trace_Dualisation} are used for bookkeeping. Notation \ref{NTN.AF_Cstar_Trace_Dualisation} fixes conventions. We use the modified standard pairing, in particular their flat and sharp operators as per Definition \ref{DFN.Wstar_Trace_MSP_Musical} and Remark \ref{REM.Wstar_Trace_MSP}.\par
Let $(A,\tau)$ be a tracial AF-$C^{*}$-algebra. For all $j\in\mathbb{N}$, we have $A_{j}\cong A_{j}^{*}$ via musical isomorphisms. Let $\mathfrak{A}$ as per Definition \ref{DFN.AF_Cstar_Trace_Dualisation}. Note $A_{0}\subset\mathfrak{A}$.

\begin{dfn}\label{DFN.AF_Cstar_Trace_Dualisation}
For all Hilbert subspaces $V\subset L^{2}(A,\tau)$, let $\pi_{V}^{A}:L^{2}(A,\tau)\longrightarrow V$ be the Hilbert space projection. Let $\mathfrak{A}=A$ or $\mathfrak{A}=L^{p}(A,\tau)$ for $p\in [1,\infty]$.

\begin{itemize}
\item[1)] For all $j\in\mathbb{N}$, let $\iota_{j}^{A}:A_{j}\longrightarrow\mathfrak{A}$ be the inclusion and set $\pi_{j}^{A}:=\pi_{A_{j}}^{A}$. \phantom{\bigg)}

\item[2)] For all $j\leq k$ in $\mathbb{N}$, let $\iota_{kj}^{A}:A_{j}\longrightarrow A_{k}$ be the inclusion and set $\pi_{jk}^{A}:=\pi_{A_{j}}^{A_{k}}$. \phantom{\bigg)}

\item[3)] For all $j\leq k$ in $\mathbb{N}$, we define the $j$-th inclusion and $j$-th restriction\phantom{\bigg)}

\begin{align}\label{EQ.DFN.AF_Cstar_Trace_Dualisation_1}
\incj:=\flat\circ\iota_{j}^{A}\circ\sharp:A_{j}^{*}\longrightarrow\mathfrak{A}^{*},\ \resj:=\big(\iota_{j}^{A}\big)^{*}:\mathfrak{A}^{*}\longrightarrow A_{j}^{*},
\end{align}

\begin{reapply}
\end{reapply}

as well as the $kj$-inclusion and $jk$-restriction\phantom{\bigg)}

\begin{align}\label{EQ.DFN.AF_Cstar_Trace_Dualisation_2}
\inckj:=\big(\pi_{jk}^{A}\big)^{*}:A_{j}^{*}\longrightarrow A_{k}^{*},\ \resjk:=\big(\iota_{kj}^{A}\big)^{*}:A_{k}^{*}\longrightarrow A_{j}^{*}.
\end{align}

\begin{reapply}
\end{reapply}

\end{itemize}
\end{dfn}

\begin{prp}\label{PRP.AF_Cstar_Trace_Dualisation_I}
All inclusion and restriction maps in Definition \ref{DFN.AF_Cstar_Trace_Dualisation} are bounded linear, positivity-preserving, as well as injective, resp.~surjective. They furthermore satisfy the following.

\begin{itemize}
\item[1)] All inclusion and restriction maps in Definition \ref{DFN.AF_Cstar_Trace_Dualisation} are $w^{*}$-continuous.

\item[2)] For all indices, $\res\circ\inc=\id$. For all $j\leq k$ in $\mathbb{N}$, we have $A_{j}^{*}\subset A_{k}^{*}\subset A^{*}$ as partially ordered Banach spaces and

\begin{itemize}
\item[2.1)] $\inckj=\flat\circ\iota_{kj}^{A}\circ\sharp$ and $\resjk=\flat\circ\pi_{jk}^{A}\circ\sharp$,

\item[2.2)] $\inckj=\resk\circ\incj$ and $\resjk=\resj\circ\inck$.
\end{itemize}

\begin{reapply}
\end{reapply}

\end{itemize}
\end{prp}
\begin{proof}
Bounded linearity is immediate. Since $A_{0}\subset A$ is $\|.\|_{A}$-dense, testing on $A_{0}$ shows continuity in each case. We directly verify all remaining claims.
\end{proof}

\begin{ntn}\label{NTN.AF_Cstar_Trace_Dualisation}
For all $j\leq k$ in $\mathbb{N}$, the following holds. We suppress $\incj$ and $\inckj$ on $A_{j}^{*}$. We neither distinguish $\pi_{j}^{A}$ and $\pi_{jk}^{A}$ on $A_{k}$, nor $\resj$ and $\resjk$ on $A_{k}^{*}$.
\end{ntn}


\pagebreak


\begin{dfn}\label{DFN.AF_Cstar_Trace_Dualisation_Restriction}
Let $j\in\mathbb{N}$ and $p\in [1,\infty]$.

\begin{itemize}
\item[1)] For all $\mu\in A^{*}$, set $\mu_{j}:=\resj\mu\in A_{j}^{*}$.

\item[2)] For all $x\in L^{p}(A,\tau)$, set $x_{j}:=\sharp\resj x^{\flat}\in A_{j}$.
\end{itemize}
\end{dfn}

\begin{prp}\label{PRP.AF_Cstar_Trace_Dualisation_II}\hspace{1cm}
\begin{itemize}
\item[1)] For all $\mu\in A^{*}$, we have

\begin{itemize}
\item[1.1)] $\|\mu\|_{A^{*}}=\sup_{j\in\mathbb{N}}\|\mu_{j}\|_{A^{*}}=\lim_{j\in\mathbb{N}}\|\mu_{j}\|_{A^{*}}$,

\item[1.2)] $\mu=w^{*}$-$\lim_{j\in\mathbb{N}}\mu_{j}$.
\end{itemize}

\begin{reapply}
\end{reapply}

\item[2)] Let $p\in [1,\infty]$. For all $x\in L^{p}(A,\tau)$, we have

\begin{itemize}
\item[2.1)] $\|x\|_{p}=\sup_{j\in\mathbb{N}}\|x_{j}\|_{p}=\lim_{j\in\mathbb{N}}\|x_{j}\|_{p}$,

\item[2.2)] $x=w^{*}$-$\lim_{j\in\mathbb{N}}x_{j}$.
\end{itemize}

\begin{reapply}
\end{reapply}

\item[3)] For all $x\in L^{\infty}(A,\tau)$, we have $x=\bds$-$\lim_{j\in\mathbb{N}}x_{j}=\bdw$-$\lim_{j\in\mathbb{N}}x_{j}$.
\end{itemize}
\end{prp}
\begin{proof}
We directly verify $1.1)$ and $2.1)$. They ensure uniform boundedness upon testing for $1.2)$, $2.2)$ and $3)$ on $A_{0}$. We conclude by density in each case.
\end{proof}

\begin{rem}\label{REM.AF_Cstar_Trace_Dualisation}
Let $j\in\mathbb{N}$. For all $x\in L^{2}(A,\tau)$, we have $x_{j}=\pi_{j}^{A}(x)$. Note Theorem \ref{THM.AF_Cstar_Bimodule_CLRA_SR} furthermore generalises strong convergence as per $3)$ in Proposition \ref{PRP.AF_Cstar_Trace_Dualisation_II} to strong resolvent convergence of positive and suitably integrable measurable operators under canonical left-~and right-actions of AF-$C^{*}$-bimodules.
\end{rem}

Following Notation \ref{NTN.AF_Cstar_Trace_Dualisation}, we treat restriction as single operation even if domains vary or identified with duals via musical isomorphisms. We use Notation \ref{NTN.PO}. For all $j\in\mathbb{N}$, $A_{j,h}^{*}$ denotes the real and $A_{j,+}^{*}$ the positive elements in $A_{j}^{*}$.

\begin{prp}\label{PRP.AF_Cstar_Trace_Dualisation_III}
For all $j\leq k$ in $\mathbb{N}$, we have

\begin{itemize}
\item[1)] $A_{j,+}^{*}\subset A_{k,+}^{*}\subset A_{+}^{*}$ and $\SII(A_{j})\subset\SII(A_{k})\subset\mathcal{S}^{\NI}(A)$,

\item[2)] $\resj\big(A_{+}^{*}\big)\subset A_{j,+}^{*}$ and $\resj\big(A_{k,+}^{*}\big)\subset A_{j,+}^{*}$.
\end{itemize}
\end{prp}
\begin{proof}
For all $j\in\mathbb{N}$ and $\mu\in A_{j}^{*}$, get $\lim_{k\in\mathbb{N}}\mu(1_{A_{k}})=\|\mu\|_{A_{j}^{*}}$. Apply Proposition \ref{PRP.AF_Cstar_Trace_Dualisation_I}.
\end{proof}

Semi-finite $W^{*}$-subalgebras have unique noncommutative conditional expectations as per Remark \ref{REM.Cstar_Trace_Abstract_Dualisation}. In the unital finite-dimensional case, they are averages of unitary conjugations \cite{COL.Car.2010.Quantum_Entropy}\cite{ART.Hia_Pet.2012.Quasi_Entropy_I}\cite{ART.Hia_Pet.2013.Quasi_Entropy_II} as per Proposition \ref{PRP.AF_Cstar_Trace_NCE_I}. Proposition \ref{PRP.AF_Cstar_Trace_NCE_II} generalises to the non-unital finite-dimensional one. In Subsection \ref{SSEC.NCDS_NCD_QE}, Lemma \ref{LEM.QE_Fin_II} moreover uses Proposition \ref{PRP.AF_Cstar_Trace_NCE_II} to show monotonicity of quasi-entropies.\par


\pagebreak


Assume $A$ is finite-dimensional. Let $N\subset A$ be a $C^{*}$-subalgebra. The commutant $N'\subset A$ of $N$ in $A$ is a $C^{*}$-algebra. The unitaries $\UII(A)$ of $A$ are a compact group, hence $\UII\lc{}N'\rc{}=\UII(A)\cap N'$ is one. We know $\UII\lc{}N'\rc{}=\UII\lc{}N[1_{A}]'\rc$ since $N'=N[1_{A}]'$. We therefore have $1_{N}^{\perp}=1_{A}-1_{N}$ and $N[1_{A}]=N\oplus\langle 1_{N}^{\perp}\rangle_{\mathbb{C}}$ using direct sum of $C^{*}$-algebras. Finally, we use the rescaling map $\kappa_{N}^{A}:A\longrightarrow\mathbb{C}$ \lc{}cf.~Definition \ref{DFN.Wstar_Trace_NCE_Kappa}\rc{}.

\begin{prp}\label{PRP.AF_Cstar_Trace_NCE_I}
Assume $A$ is finite-dimensional. Let $N\subset A$ be a $C^{*}$-subalgebra and $\nu_{N}$ the Haar probability measure on $\UII\lc{}N'\rc$. The noncommutative conditional expectation from $A$ to $N[1_{A}]$ is given by $\pi_{N[1_{A}]}^{A}(x)=\int_{\UII\lc{}N'\rc{}}uxu^{*}d\nu_{N}$ for all $x\in A$.
\end{prp}
\begin{proof}
We have $N[1_{A}]=\UII\lc{}N[1_{A}]'\rc{}'=\UII\lc{}N'\rc{}'$ \lc{}cf.~Proposition \ref{PRP.Wstar_Generated}\rc{}. For all $x\in A$, set $P(x):=\int_{\UII\lc{}N'\rc{}}uxu^{*}d\nu_{N}(u)$. Note transformation of Haar measures under group actions implies $P(x)\in\UII\lc{}N'\rc{}'=N[1_{A}]$. Using uniqueness of the noncommutative conditional expectation from $A$ to $N[1_{A}]$ \lc{}cf.~Definition \ref{DFN.Wstar_Trace_NCE}\rc{}, we directly verify our claim.
\end{proof}

\begin{prp}\label{PRP.AF_Cstar_Trace_NCE_II}
Assume $A$ is finite-dimensional. Let $N\subset A$ be a $C^{*}$-subalgebra. We have $\pi_{N}^{A}=\pi_{N[1_{A}]}^{A}-\kappa_{N}^{A}1_{N}^{\perp}$. For all $x\in A$, this Hilbert space projection is given by

\begin{align}\label{EQ.PRP.AF_Cstar_Trace_NCE_II_1}
\pi_{N}^{A}(x)=
\begin{cases}
\int_{\UII\lc{}N'\rc{}}uxu^{*}d\nu_{N}+\tau(1_{N}^{\perp})^{-1}\tau\vstretch{0.9375}{\bigg(}\pi_{\langle 1_{N}^{\perp}\rangle_{\mathbb{C}}}^{A}(x)\vstretch{0.9375}{\bigg)}\cdot 1_{N}^{\perp} & \If\ 1_{A}\neq 1_{N}, \\
\int_{\UII\lc{}N'\rc{}}uxu^{*}d\nu_{N} & \Else.
\end{cases}
\end{align}
\end{prp}
\begin{proof}
Apply Proposition \ref{PRP.AF_Cstar_Trace_NCE_I} and $1)$ in Proposition \ref{PRP.Wstar_Trace_NCE_II}.
\end{proof}


\subsubsection*{Definition using local $^{*}$-homomorphisms}

We use local $^{*}$-homomorphisms to define AF-$C^{*}$-bimodule actions. In addition, Lemma \ref{LEM.AF_Cstar_Local_Hom} and Corollary \ref{COR.AF_Cstar_Local_Hom} show local $^{*}$-homomorphisms extend to a $^{*}$-homomorphism of spaces of measurable operators s.t.~$L^{p}$-norms are preserved. Definition \ref{DFN.AF_Cstar_Bimodule} gives AF-$C^{*}$-bimodules.

\begin{dfn}\label{DFN.AF_Cstar_Local_Hom}
Let $(A,\tau)$ and $(B,\omega)$ be tracial AF-$C^{*}$-algebras. Let $\phi:A\longrightarrow B$ be a $^{*}$-homomorphism. 

\begin{itemize}
\item[1)] For all $j\in\mathbb{N}$ s.t.~$\phi(A_{j})\subset B_{j}$, set $\phi_{j}:=\phi\vert_{A_{j}}:(A_{j},\|.\|_{\tau})\longrightarrow (B_{j},\|.\|_{\omega})$ and $\phi_{j}^{*}:=\big(\phi_{j}\big)^{*}$ for its adjoint.

\item[2)] We say that $\phi$ satisfies

\begin{itemize}
\item[2.1)] local unitality if $\phi(1_{A_{j}})=1_{B_{j}}$ for all $j\in\mathbb{N}$,

\item[2.2)] locality if $\phi(A_{j})\subset B_{j}$ and $\phi_{k}^{*}(B_{j})\subset A_{j}$ for all $j\leq k$ in $\mathbb{N}$,

\item[2.3)] extendability if $\sup_{j\in\mathbb{N}}\|\phi_{j}^{*}(1_{B_{j}})\|_{A},\sup_{j\in\mathbb{N}}\|\phi_{j}^{*}\|_{\BII(B_{j},A_{j})}<\infty$.
\end{itemize}

\begin{reapply}
\end{reapply}

\item[3)] We call $\phi$ local if it satisfies locality, local unitality and extendability.
\end{itemize}
\end{dfn}

\begin{bsp}
For all tracial AF-$C^{*}$-algebras $(A,\tau)$, its identity map $\id_{A}$ is local.
\end{bsp}

\begin{prp}\label{PRP.AF_Cstar_Local_Unitality}
Let $(A,\tau)$ be a tracial AF-$C^{*}$-algebra s.t.~$\tau<\infty$. If $T\in\BII\lc{}L^{2}(A,\tau)\rc$ s.t.~$T(1_{A_{j}})=1_{A_{j}}$ for all $j\in\mathbb{N}$, then $T(1_{A})=1_{A}$.
\end{prp}
\begin{proof}
Since $\tau<\infty$, get $1_{A}\in L^{2}(A,\tau)$ and $A_{0}\subset L^{2}(A,\tau)$. Thus $2)$ in Proposition \ref{PRP.AF_Cstar_Unit} implies  $1_{A}=\s$-$\lim_{j\in\mathbb{N}}1_{A_{j}}$, hence $1_{A}=\|.\|_{\tau}$-$\lim_{j\in\mathbb{N}}1_{A_{j}}$.
\end{proof}

Let $(A,\tau)$ and $(B,\omega)$ be tracial AF-$C^{*}$-algebras. Note $2)$ in Proposition \ref{PRP.AF_Cstar_Trace_I} shows $A_{0}\subset L^{2}(A,\tau)$ is $\|.\|_{\tau}$-dense and $B_{0}\subset L^{2}(B,\omega)$ is $\|.\|_{\omega}$-dense. We use such density for Equation \ref{EQ.SSEC.NCDS_AF_BIM_2}. Let $\phi:A\longrightarrow B$ be a $^{*}$-homomorphism. If $\phi$ satisfies locality, then 

\begin{align}\label{EQ.SSEC.NCDS_AF_BIM_1}
\restr{0.925}{\phi_{k}^{*}}{B_{j}}=\phi_{j}^{*}
\end{align}

\noindent for all $j\leq k$ in $\mathbb{N}$. Assume $\phi$ is local. For all $x\in A\cap L^{2}(A,\tau)$ and $u\in L^{2}(B,\omega)$, we use density and extendability to get $\phi(x)\in L^{2}(B,\omega)$ and

\begin{align}\label{EQ.SSEC.NCDS_AF_BIM_2}
\lgl\phi(x),u\rgl_{\omega}\leq \|x\|_{\tau}\cdot \sup_{j\in\mathbb{N}}\hspace{0.025cm} \|\phi_{j}^{*}\|_{\BII(B_{j},A_{j})}\hspace{0.025cm} \|u\|_{\omega}<\infty.
\end{align}

\noindent Equation \ref{EQ.SSEC.NCDS_AF_BIM_2} yields extension $\phi^{2}\in\BII(L^{2}(A,\tau),L^{2}(B,\omega))$ of $\phi$ with norm

\begin{align}\label{EQ.SSEC.NCDS_AF_BIM_3}
\|\phi\|_{2}:=\|\phi^{2}\|_{\BII(L^{2}(A,\tau),L^{2}(B,\omega))}\leq\sup_{j\in\mathbb{N}}\hspace{0.025cm} \|\phi_{j}^{*}\|_{\BII(B_{j},A_{j})}.
\end{align}

\begin{dfn}\label{DFN.AF_Cstar_Local_Hom_L2}
Let $(A,\tau)$ and $(B,\omega)$ be tracial AF-$C^{*}$-algebras. Let $\phi:A\longrightarrow B$ be a local $^{*}$-homomorphism. We call $\phi^{2}\in\BII(L^{2}(A,\tau),L^{2}(B,\omega))$ the $L^{2}$-extension of $\phi$. Let $\phi^{2,*}:=\big(\phi^{2}\big)^{*}$ be its adjoint.
\end{dfn}

\begin{prp}\label{PRP.AF_Cstar_Local_Hom_I}
Let $(A,\tau)$ and $(B,\omega)$ be tracial AF-$C^{*}$-algebras. Let $\phi:A\longrightarrow B$ be a local $^{*}$-homomorphism.

\begin{itemize}
\item[1)] $\phi^{*}\circ\flat\vert_{B_{0}}=\flat\circ\lc\restr{0.925}{\phi^{2,*}}{B_{0}}\rc$ using Banach dual $\phi^{*}:B^{*}\longrightarrow A^{*}$.

\item[2)] $\phi^{2,*}$ is positivity-preserving.

\item[3)] For all $j\in\mathbb{N}$, we have

\begin{itemize}
\item[3.1)] $\restr{0.925}{\phi^{2,*}}{B_{j}}=\phi_{j}^{*}$ and $\big[\pi_{j}^{A},\phi^{2}\big]=0$,

\item[3.2)] $\dblv{}\phi^{2,*}(u)\dblv_{A}\leq \|\phi_{j}^{*}(1_{B_{j}})\|_{A}\|u\|_{B}$ for all $u\in B_{j,h}$.
\end{itemize}

\begin{reapply}
\end{reapply}

\end{itemize}
\end{prp}
\begin{proof}
We have $3.1)$ by locality. Using $3.1)$, we directly verify $1)$ by testing on $A_{0}$ in each case. Then $1)$ shows $2)$ since $\phi$ is a $^{*}$-homomorphism and $\flat$ is positivity-preserving. For all Hilbert spaces $H$, $T\in\BII(H)_{h}$ and $C\geq 0$, we have $\| T\|_{\BII(H)}\leq C$ if and only if $-CI\leq T\leq CI$. Using the latter, note $2)$ and $3.1)$ show $3.2)$ immediately.
\end{proof}

\begin{lem}\label{LEM.AF_Cstar_Local_Hom}
Let $(A,\tau)$ and $(B,\omega)$ be tracial AF-$C^{*}$-algebras. Let $\phi:A\longrightarrow B$ be a local $^{*}$-homomorphism.

\begin{itemize}
\item[1)] There exists positivity-preserving $w^{*}$-continuous $\phi^{1}\in\BII\lc{}L^{1}(A,\tau),L^{1}(B,\omega)\rc$ with\linebreak norm $\|\phi\|_{1}:=\dblv{}\phi^{1}\dblv{}\leq 2\sup_{j\in\mathbb{N}}\|\phi_{j}^{*}(1_{B_{j}})\|_{A}$ extending $\phi$. Let $\phi^{1,*}:=\big(\phi^{1}\big)^{*}$ be its\linebreak Banach dual. We have

\begin{itemize}
\item[1.1)] $\omega\lc\phi^{1}(x)^{*}u\rc{}=\tau\lc{}x^{*}\phi^{1,*}(u)\rc$ for all $x\in L^{1}(A,\tau)$ and $u\in L^{\infty}(B,\omega)$,

\item[1.2)] $\phi_{0}^{*}:=\restr{0.925}{\phi^{1,*}}{B_{0}}=\restr{0.925}{\phi^{2,*}}{B_{0}}$.
\end{itemize}

\begin{reapply}
\end{reapply}

\item[2)] There exists normal unital $^{*}$-homomorphism $\phi^{\infty}:L^{\infty}(A,\tau)\longrightarrow L^{\infty}(B,\omega)$ with norm $\|\phi\|_{\infty}:=\dblv{}\phi^{\infty}\dblv{}=1$ extending $\phi$. Let $\phi^{\infty,*}:=\big(\phi^{\infty}\big)^{*}$ be its Banach dual. We have $\phi^{\infty,*}\circ\flat\vert_{B_{0}}=\flat\circ\phi_{0}^{*}$.

\item[3)] For all $j\in\mathbb{N}$, $\phi^{1}(x_{j})=\phi^{1}(x)_{j}$ for all $x\in L^{1}(A,\tau)$.
\end{itemize}
\end{lem}
\begin{proof}
Note $\lc\sigma\textrm{-}\rc$weak-~and $w^{*}$-convergence coincide on bounded sets \lc{}cf.~Lemma II.2.5 in \cite{BK.Tak.1979.OpAlg_I} and Proposition \ref{PRP.Wstar_Equivalence}\rc{}. We use bounded strong and bounded weak convergence \lc{}cf.~Definition \ref{DFN.Wstar_BdCon_I} and Remark \ref{REM.Wstar_BdCon_I}\rc{}. In particular, multiplication in $W^{*}$-algebras is bounded strongly continuous \lc{}cf.~Remark \ref{REM.Wstar_BdCon_II}\rc{}. We know Proposition \ref{PRP.Wstar_BdCon} applies to $A_{0}\subset L^{\infty}(A,\tau)$ and $B_{0}\subset L^{\infty}(B,\omega)$ by $\sigma$-weak density.\par
We show $1)$. Let $x\in A_{0}$ and $u\in L^{\infty}(B,\omega)$. If $\|u\|_{\infty}=1$, then Proposition \ref{PRP.Wstar_BdCon} yields $\lset{}u_{k}\rset_{k\in K}\subset B_{0}$ s.t.~$\sup_{k\in K}\dblv{}u_{k}\dblv_{B}\leq 1$ and $u=w^{*}$-$\lim_{k\in K}u_{k}$. If we furthermore apply $3.2)$ in Proposition \ref{PRP.AF_Cstar_Local_Hom_I} to $\RE(u_{k})$ and $\IM(u_{k})$ for all $k\in K$ \lc{}cf.~Proposition \ref{PRP.Wstar_NCI_IV}\rc{}, then we calculate

\begin{align*}
\babsv{1}{\omega\lc\phi(x)^{*}u\rc{}} & = \limsup_{k\in K}\hspace{0.025cm} \babsv{1}{\omega\lc{}x^{*}\phi^{2,*}(u_{k})\rc{}} \phantom{\bigg)} \\
& \leq\ \limsup_{k\in K}\hspace{0.025cm} \babsv{1}{\omega\lc{}x^{*}\phi^{2,*}\big(\RE(u_{k})\big)\rc{}}+\limsup_{k\in K}\hspace{0.025cm} \babsv{1}{\omega\lc{}x^{*}\phi^{2,*}\big(\IM(u_{k})\big)\rc{}} \phantom{\bigg)} \\
& \leq \|x\|_{1}\cdot 2\sup_{j\in\mathbb{N}}\hspace{0.025cm} \|\phi_{j}^{*}(1_{B_{j}})\|_{A}. \phantom{\bigg)}
\end{align*}

\noindent Using the above calculation, linearity and extendability of $\phi$ let us estimate

\begin{align}\label{EQ.LEM.AF_Cstar_Local_Hom_1}
\babsv{1}{\omega\lc\phi(x)^{*}u\rc{}}\leq \|x\|_{1}\cdot 2\sup_{j\in\mathbb{N}}\hspace{0.025cm} \|\phi_{j}^{*}(1_{B_{j}})\|_{A}\|u\|_{\infty}<\infty.
\end{align}

\noindent Equation \ref{EQ.LEM.AF_Cstar_Local_Hom_1} yields extension $\phi^{1}\in\BII\lc{}L^{1}(A,\tau),L^{1}(B,\omega)\rc$ of $\phi$ with norm estimate as claimed. Using boundedness and $3.1)$ in Proposition \ref{PRP.AF_Cstar_Local_Hom_I}, we directly verify $1.1)$ by testing on $A_{0}$ and $B_{0}$. Note $1.1)$ implies $1.2)$. Using properties of the modified standard pairing \lc{}cf.~Proposition \ref{PRP.Wstar_NCI_VI}\rc{}, we additionally see $1.1)$ implies positivity-preservation and $w^{*}$-continuity of $\phi^{1}$. Altogether, get $1)$.\par


\pagebreak


We show $2)$. Using $1.1)$, traciality lets us calculate

\begin{align}\label{EQ.LEM.AF_Cstar_Local_Hom_2}
\omega\lc\lc\phi(x)-\phi(y)\rc^{*}u\rc{}=\tau\lc\lc{}x^{*}-y^{*}\rc\phi^{1,*}(u)\rc{}=\overline{\tau\lc\lc{}x-y\rc\phi^{1,*}(u)^{*}\rc{}}
\end{align}

\noindent for all $x,y\in A_{0}$ and $u\in L^{\infty}(B,\omega)$. For all $x\in L^{\infty}(A,\tau)$, Proposition \ref{PRP.Wstar_BdCon} shows there exists bounded net $\{x_{k}\}_{k\in K}\subset A_{0}$ s.t.~$x=w^{*}$-$\lim_{k\in K}x_{k}$. Using the latter in order to test on $A_{0}$, Equation \ref{EQ.LEM.AF_Cstar_Local_Hom_2} yields positivity-preserving and $w^{*}$-continuous linear extension $\phi^{\infty}:L^{\infty}(A,\tau)\longrightarrow L^{\infty}(B,\omega)$ of $\phi$ by boundedness. For all $u\in B_{0}$, $2)$ in Proposition \ref{PRP.AF_Cstar_Local_Hom_I} implies $\phi^{2,*}\lc{}uu^{*}\rc{}=a_{u}^{2}\geq 0$ for a self-adjoint $a_{u}\in A_{0}$. For all $x,y\in A_{0}$ and $u\in B_{0}$, get

\begin{align}\label{EQ.LEM.AF_Cstar_Local_Hom_3}
\dblv{}\lc\phi(x)-\phi(y)\rc{}u\dblv_{\omega}^{2}=\lgl\phi\lc\lc{}x-y\rc^{*}\lc{}x-y\rc\rc{},uu^{*}\rgl_{\omega}=\dblv{}\lc{}x-y\rc{}a_{u}\dblv_{\tau}^{2}.
\end{align}

\noindent We know $B_{0}\subset L^{2}(B,\omega)$ is $\|.\|_{\omega}$-dense. Thus Equation \ref{EQ.LEM.AF_Cstar_Local_Hom_3} shows $\phi^{\infty}$ is bounded strongly convergent, hence $\phi$ is a $^{*}$-homomorphism. Ergo $\phi^{\infty}$ is normal by Proposition \ref{PRP.Wstar_Normal}, as well as unital by local unitality and $2)$ in Proposition \ref{PRP.AF_Cstar_Unit}. Get $\dblv{}\phi^{\infty}\dblv{}=1$. Using $1)$ in Proposition \ref{PRP.AF_Cstar_Local_Hom_I}, we directly verify $\phi^{\infty,*}\circ\flat\vert_{B_{0}}=\flat\circ\phi_{0}^{*}$. Altogether, get $2)$.\par
We show $3)$. Following Remark \ref{REM.Wstar_Trace_MSP}, noncommutative $L^{1}$-spaces are subsets of their Banach double dual spaces. Following Notation \ref{NTN.AF_Cstar_Trace_Dualisation}, get $3)$ if

\begin{align}\label{EQ.LEM.AF_Cstar_Local_Hom_4}
\phi^{1}(x_{j})=\phi^{1}\lc\sharp\resj x^{\flat}\rc{}=\sharp\resj\phi^{1}(x)^{\flat}=\phi^{1}(x)_{j}
\end{align}

\noindent for all $x\in L^{1}(A,\tau)$ and $j\in\mathbb{N}$. Using $1.1)$ and $\phi^{2,*}\lc{}B_{0}\rc\subset A_{0}$, get Equation \ref{EQ.LEM.AF_Cstar_Local_Hom_4} at once.
\end{proof}

\begin{cor}\label{COR.AF_Cstar_Local_Hom}
Let $(A,\tau)$ and $(B,\omega)$ be tracial AF-$C^{*}$-algebras. Let $\phi:A\longrightarrow B$ be a local $^{*}$-homomorphism. There exists unital $^{*}$-homomorphism $\phi:L^{0}(A,\tau)\longrightarrow L^{0}(B,\omega)$ continuous in measure topologies extending $\phi$.
\end{cor}
\begin{proof}
We use uniform structures \lc{}cf.~Equation \ref{EQ.SSEC.B_SMO_NCI_1}\rc{}. If $p\in L^{\infty}(A,\tau)$ is a projection, then $\phi\lc{}p\rc\in L^{\infty}(B,\omega)$ is a projection and $\phi\lc{}p^{\perp}\rc{}=\phi\lc{}p\rc^{\perp}$ by $2)$ in Lemma \ref{LEM.AF_Cstar_Local_Hom}. If furthermore $p^{\perp}\in L^{1}(A,\tau)$, then $\phi\lc{}p\rc^{\perp}\in L^{1}(B,\omega)$ by $1)$ in Lemma \ref{LEM.AF_Cstar_Local_Hom}. Let $\varepsilon,\delta>0$. If $x\in L^{\infty}(A,\tau)$ and $p\in L^{\infty}(A,\tau)$ projection s.t.~$\|xp\|_{\infty}<\varepsilon$ and $\tau\lc{}p^{\perp}\rc{}<\delta$, then $\dblv{}\phi(x)\phi\lc{}p\rc\dblv_{\infty}\leq \|xp\|_{\infty}<\varepsilon$ by $2)$ in Lemma \ref{LEM.AF_Cstar_Local_Hom} and $\tau\lc\phi\lc{}p\rc^{\perp}\rc\leq \|\phi\|_{1}\tau\lc{}p^{\perp}\rc{}<\|\phi\|_{1}\delta$ by $1)$ in Lemma \ref{LEM.AF_Cstar_Local_Hom}.\par
For all $\varepsilon,\delta>0$, get $\phi\lc{}N\lc\varepsilon,\delta\rc\rc\subset N\lc\varepsilon,\|\phi\|_{1}\delta\rc$. Thus $\phi$ maps bounded Cauchy nets to bounded Cauchy nets in measure topologies, hence extends as claimed. For this, note algebra involution and multiplication in spaces of measurable operators are continuous in measure topology on bounded subsets \lc{}cf.~Theorem IX.2.2 in \cite{BK.Tak.2003.OpAlg_II} or \cite{ART.Nel.1974.Wstar_Integration}\rc{}.
\end{proof}

\begin{ntn}\label{NTN.AF_Cstar_Local_Hom}
All extensions of local $^{*}$-homomorphisms as discussed above coincide on intersections of domains. Unless stated otherwise, we do not discern extensions. For all local $^{*}$-homomorphisms $\phi$, we write $\phi$ for extensions and $\phi^{*}$ for their adjoints.
\end{ntn}


\pagebreak


Definition \ref{DFN.AF_Cstar_Bimodule} gives AF-$C^{*}$-bimodules. Proposition \ref{PRP.AF_Cstar_Bimodule} moreover shows they induce symmetric $W^{*}$-bimodules as per Definition \ref{DFN.CWstar_Bimodule}, i.e.~as per Definition \ref{DFN.CWstar_Bimodule_Induced} in all further use below. Let $\phi,\bpsi:A\longrightarrow B$ be local $^{*}$-homomorphisms. We define bounded $A$-bimodule action on $B$ by setting

\begin{align}\label{EQ.SSEC.NCDS_AF_BIM_4}
xuy:=\phi(x)u\bpsi(y)    
\end{align}

\noindent for all $x,y\in A$ and $u\in B$. Applying $2)$ in Lemma \ref{LEM.AF_Cstar_Local_Hom}, we extend Equation \ref{EQ.SSEC.NCDS_AF_BIM_4} to a normal, unital and bounded $L^{\infty}(A,\tau)$-bimodule action on $L^{2}(B,\omega)$. Symmetry requires anti-linear involution, with algebra involution the canonical example.

\begin{dfn}
Let $(A,\tau)$ be a tracial AF-$C^{*}$-algebra. We call anti-linear isometric involution $\gamma:L^{2}(A,\tau)\longrightarrow L^{2}(A,\tau)$ local if $\gamma(A_{j})\subset A_{j}$ and $\gamma(1_{A_{j}})=1_{A_{j}}$ for all $j\in\mathbb{N}$.
\end{dfn}

\begin{bsp}\label{BSP.AF_Cstar_Bimodule_Canonical}
For all tracial AF-$C^{*}$-algebras $(A,\tau)$, note the algebra involution on $A$ itself extends to a local anti-linear isometric involution $\Adj:L^{2}(A,\tau)\longrightarrow L^{2}(A,\tau)$ since $A_{0}\subset\mathfrak{m}_{\tau}$ \lc{}cf.~Proposition \ref{PRP.Wstar_NCI_I}\rc{}.
\end{bsp}

\begin{dfn}\label{DFN.AF_Cstar_Bimodule}
Let $(A,\tau)$ and $(B,\omega)$ be tracial AF-$C^{*}$-algebras. Let $\phi,\bpsi:A\longrightarrow B$ be local $^{*}$-homomorphisms. Let $\gamma:L^{2}(B,\omega)\longrightarrow L^{2}\lc{}B,\gamma\rc$ be a local anti-linear isometric involution.

\begin{itemize}
\item[1)] The AF-$A$-bimodule action given by Equation \ref{EQ.SSEC.NCDS_AF_BIM_4} is called the $(\phi,\bpsi)$-action of $A$ on $B$. Its extension to $L^{\infty}(A,\tau)$ acting on $L^{2}(B,\omega)$ is called normal extension.

\item[2)] We say that the $(\phi,\bpsi)$-action satisfies $\gamma$-symmetry if

\begin{align}\label{EQ.DFN.AF_Cstar_Bimodule_1}
\gamma\lc\phi(x)u\bpsi(y)\rc{}=\phi(y^{*})\gamma(u)\bpsi(x^{*})    
\end{align}

\begin{reapply}
\end{reapply}

\noindent for all $x,y\in A$ and $u\in B$.

\item[3)] We call $(\phi,\bpsi,\gamma)$ an AF-$A$-bimodule structure on $B$, or AF-$A$-bimodule over $B$ if the $(\phi,\bpsi)$-action satisfies $\gamma$-symmetry. We further call $(\phi,\bpsi,\gamma)$ an AF-$C^{*}$-bimodule.

\item[4)] Let $(\phi,\bpsi,\gamma)$ be an AF-$A$-bimodule structure on $B$. For all $j\in\mathbb{N}$, we consider tracial AF-$C^{*}$-algebras $(A_{j},\tau)$ and $(B_{j},\omega)$ as per Definition \ref{DFN.AF_Cstar_Trace_Restriction}. We furthermore call $(\phi_{j},\bpsi_{j},\gamma_{j}):=(\phi\vert_{A_{j}},\bpsi\vert_{A_{j}},\gamma\vert_{A_{j}})$ the induced AF-$A_{j}$-bimodule structure on $B_{j}$.

\item[5)] Assume $\phi=\psi=\id_{A}$ and further $\gamma=\Adj$ as per Example \ref{BSP.AF_Cstar_Bimodule_Canonical} for $A$ as anti-linear involution. We call $\lc\id_{A},\id_{A},\Adj\rc$ the canonical AF-$A$-bimodule structure on $A$.
\end{itemize}
\end{dfn}

\begin{prp}\label{PRP.AF_Cstar_Bimodule_Well_Defined}
Let $(A,\tau)$ and $(B,\omega)$ be tracial AF-$C^{*}$-algebras. If $(\phi,\bpsi,\gamma)$ is an AF-$A$-bimodule structure on $B$, then we have AF-$A_{j}$-bimodule structure $(\phi_{j},\bpsi_{j},\gamma_{j})$ on $B_{j}$ for all $j\in\mathbb{N}$. If $\phi=\bpsi=\id_{A}$, then we have AF-$A$-bimodule structure $\lc\id_{A},\id_{A},\Adj\rc$ on $A$.
\end{prp}
\begin{proof}
By construction of either case.
\end{proof}


\pagebreak


\begin{dfn}\label{DFN.CWstar_Bimodule}
Let $A$ be a $C^{*}$-algebra and $\phi,\bpsi:A\longrightarrow\BII(H)$ $^{*}$-homomorphisms. For all $x,y\in A$, let $\lb\phi(x),\bpsi(y)\rb{}=0$. Let $H$ be a Hilbert space and $\gamma:H\longrightarrow H$ an anti-linear isometric involution. We define bounded $A$-bimodule action by setting

\begin{align}\label{EQ.DFN.CWstar_Bimodule_1}
xuy:=\phi(x)\lc\bpsi(y)(u)\rc{}
\end{align}

\noindent for all $x,y\in A$ and $u\in H$. The $A$-bimodule action given by Equation \ref{EQ.DFN.CWstar_Bimodule_1} is called the $(\phi,\bpsi)$-action of $A$ on $H$. We say that the $(\phi,\bpsi)$-action satisfies $\gamma$-symmetry if

\begin{align}\label{EQ.DFN.CWstar_Bimodule_2}
\gamma\lc{}xuy\rc{}=y^{*}\gamma(u)x^{*}    
\end{align}

\noindent for all $x,y\in A$ and $u\in H$. We call $H$ a symmetric $C^{*}$-bimodule over $A$ if the bounded $A$-bimodule action satisfies $\gamma$-symmetry. We call $H$ a symmetric $W^{*}$-bimodule if $A=M$ is a $W^{*}$-algebra, $H$ is a symmetric $C^{*}$-bimodule over $M$, and $\phi,\bpsi$ are normal unital.
\end{dfn}

\begin{prp}\label{PRP.AF_Cstar_Bimodule}
Let $(A,\tau)$ and $(B,\omega)$ be tracial AF-$C^{*}$-algebras. If $(\phi,\bpsi,\gamma)$ is an AF-$A$-bimodule structure on $B$, then $L^{2}(B,\omega)$ equipped with the normal extension of the $(\phi,\bpsi)$-action and $\gamma$ is a symmetric $W^{*}$-bimodule over $L^{\infty}(A,\tau)$.
\end{prp}
\begin{proof}
Note $(\phi,\bpsi)$-action as per Equation \ref{EQ.SSEC.NCDS_AF_BIM_4} is $\lc{}L^{\phi},R^{\bpsi}\rc$-action as per Definition \ref{DFN.CWstar_Bimodule} and Definition \ref{DFN.AF_Cstar_Bimodule_CLRA}. Thus $L^{2}(B,\omega)$ is symmetric $C^{*}$-bimodule over $A$ for $\gamma$ anti-linear involution. We extend by $2)$ in Lemma \ref{LEM.AF_Cstar_Local_Hom} and bounded strong continuity of $\gamma$.
\end{proof}

\begin{dfn}\label{DFN.CWstar_Bimodule_Induced}
Let $(A,\tau)$ and $(B,\omega)$ be tracial AF-$C^{*}$-algebras. Let $(\phi,\bpsi,\gamma)$ be an AF-$A$-bimodule structure on $B$. We equip $L^{2}(B,\omega)$ with the normal extension of the $(\phi,\bpsi)$-action and $\gamma$. We call $L^{2}(B,\omega)$ the induced symmetric $W^{*}$-bimodule of $(\phi,\bpsi,\gamma)$.
\end{dfn}


\subsection[Functional calculus for AF-$C^{*}$-bimodules]{Functional calculus for AF-$\mathbf{C}^{*}$-bimodules}\label{SSEC.NCDS_AF_FC}

We discuss canonical left-~and right-actions of AF-$C^{*}$-bimodules. Theorem \ref{THM.JFC_Compression} states sufficient conditions for compressing joint functional calculus pulled-back along such canonical left-~and right-actions to joint functional calculus of self-adjoint measurable operators. This defines the compressed pulled-back joint functional calculus of extended AF-$C^{*}$-bimodule actions. In Subsection \ref{SSEC.NCDS_NCD_Operators}, we use the latter to construct and control noncommutative division operators of positive measurable operators.


\subsubsection*{Canonical left-~and right-actions of AF-$C^{*}$-bimodules}

Tracial $W^{*}$-algebras determine canonical left-~and right-actions of their spaces of measurable operators on noncommutative $L^{2}$-space \lc{}cf.~Definition \ref{DFN.Wstar_CLRA}\rc{}. Compression of AF-$C^{*}$-bimodules uses semi-finite $W^{*}$-algebras and canonical inclusions of spaces of measurable operators as per Theorem \ref{THM.Wstar_L2Red} \lc{}cf.~Definition \ref{DFN.Wstar_Trace_SF_Subalg} and Remark \ref{REM.Wstar_L2Red}\rc{}. For details on underlying compression maps, we refer to Subsection \ref{SSEC.B_JFC_L2Red}.\par


\pagebreak


Pulling back along AF-$C^{*}$-bimodule actions defines canonical left-~and right-actions of AF-$C^{*}$-bimodules. We use the opposite algebra construction \lc{}cf.~Definition \ref{DFN.Oppalg}\rc{}. Let $(A,\tau)$ and $(B,\omega)$ be tracial AF-$C^{*}$-algebras. Let $(\phi,\bpsi,\gamma)$ be an AF-$A$-bimodule structure on $B$. Corollary \ref{COR.AF_Cstar_Local_Hom} lets us extend to unital $^{*}$-homomorphisms $\phi:L^{0}(A,\tau)\longrightarrow L^{0}(B,\omega)$ and $\bpsi:L^{0}(A,\tau)^{\op}\longrightarrow L^{0}(B,\omega)^{\op}$. We have canonical left-~and right-action

\begin{align}\label{EQ.SSEC.NCDS_AF_FC_1}
L_{L^{\infty}(B,\omega)}:L^{0}(B,\omega)\longrightarrow\UBII\lc{}L^{2}(B,\omega)\rc{},\ R_{L^{\infty}(B,\omega)}:L^{0}(B,\omega)^{\op}\longrightarrow\UBII\lc{}L^{2}(B,\omega)\rc{} 
\end{align}

\noindent of $L^{0}(B,\omega)$ on $L^{2}(B,\omega)$ \lc{}cf.~Definition \ref{DFN.Wstar_CLRA} and Definition \ref{DFN.Wstar_CLRA_Unbd_Representation}\rc{}. Moreover, we know they are unbounded faithful unital $^{*}$-representations \lc{}cf.~Corollary \ref{COR.Wstar_CLRA_III}\rc{}. For details on canonical left-~and right-actions, we refer to Subsection \ref{SSEC.B_SMO_CLRA}.

\begin{dfn}\label{DFN.AF_Cstar_Bimodule_CLRA}
Set $L^{\phi}:=L_{L^{\infty}(B,\omega)}\circ\phi$ and $R^{\bpsi}:=R_{L^{\infty}(B,\omega)}\circ\bpsi$. We thereby define canonical left-~and right-action

\begin{align}\label{EQ.DFN.AF_Cstar_Bimodule_CLRA_1}
L^{\phi}:L^{0}(A,\tau)\longrightarrow\UBII\lc{}L^{2}(B,\omega)\rc{},\ R^{\bpsi}:L^{0}(A,\tau)^{\op}\longrightarrow\UBII\lc{}L^{2}(B,\omega)\rc{}
\end{align}

\noindent of $L^{0}(A,\tau)$ on $L^{2}(B,\omega)$.
\end{dfn}

\begin{ntn}\label{NTN.AF_Cstar_Bimodule_CLRA}
For all $x\in L^{0}(A,\tau)$, we write $L_{x}^{\phi}:=L^{\phi}(x)$ and $R_{x}^{\bpsi}:=R^{\bpsi}(x)$. We suppress $\phi$ and $\bpsi$ in Definition \ref{DFN.AF_Cstar_Bimodule_CLRA} if $\phi=\bpsi=\id_{A}$.
\end{ntn}

\begin{bsp}\label{BSP.AF_Cstar_Bimodule_Canonical_CLRA}
In the setting of $5)$ in Definition \ref{DFN.AF_Cstar_Bimodule}, note Definition \ref{DFN.AF_Cstar_Bimodule_CLRA} is in fact canonical left-~and right-action of $L^{0}(A,\tau)$ on $L^{2}(A,\tau)$. 
\end{bsp}

Proposition \ref{PRP.AF_Cstar_Bimodule_CLRA_I} shows canonical left-~and right-actions as per Definition \ref{DFN.AF_Cstar_Bimodule_CLRA} are unbounded faithful unital $^{*}$-representations. Restriction to the bounded case yields induced symmetric $W^{*}$-bimodule actions. Proposition \ref{PRP.AF_Cstar_Bimodule_CLRA_II} uses bounded measurable functional calculus of self-adjoint measurable operator \lc{}cf.~Definition \ref{DFN.Wstar_CLRA_FC_III}\rc{}. The latter ensures positivity-preservation and shows parts of Lemma \ref{LEM.AF_Cstar_Bimodule_Compression_I}.

\begin{prp}\label{PRP.AF_Cstar_Bimodule_CLRA_I}
For all $x\in L^{0}(A,\tau)$, $L_{x}^{\phi}$ and $R_{x}^{\bpsi}$ are densely defined closed operators on $L^{2}(M,\tau)$. For all $x,y\in L^{0}(B,\omega)$ and $\lambda\in\mathbb{C}$, we have

\begin{itemize}
\item[1)] $L_{\lambda_{1}x+\lambda_{2}y}^{\phi}=\overline{\lambda_{1}L_{x}^{\phi}+\lambda_{2}L_{y}^{\phi}}$ and $R_{\lambda_{1}x+\lambda_{2}y}^{\bpsi}=\overline{\lambda_{1}R_{x}^{\bpsi}+\lambda_{2}R_{y}^{\bpsi}}$,

\item[2)] $L_{xy}^{\phi}=\overline{L_{x}^{\phi}L_{y}^{\phi}}$ and $R_{xy}^{\bpsi}=\overline{R_{y}^{\bpsi}R_{x}^{\bpsi}}$,

\item[3)] $L_{x^{*}}^{\phi}=\big(L_{x}^{\phi}\big)^{*}$ and $R_{x^{*}}^{\bpsi}=\big(R_{x}^{\bpsi}\big)^{*}$.
\end{itemize}
\end{prp}
\begin{proof}
Apply Corollary \ref{COR.AF_Cstar_Local_Hom} and Corollary \ref{COR.Wstar_CLRA_III}.
\end{proof}


\pagebreak


\begin{prp}\label{PRP.AF_Cstar_Bimodule_CLRA_II}
For all $x,y\in L^{0}(A,\tau)_{h}$, we have

\begin{itemize}
\item[1)] $L_{x}^{\phi},R_{y}^{\bpsi}\in\UBII\lc{}L^{2}(B,\omega)\rc_{+}$ commute strongly,

\item[2)] $L^{\phi}\lc\Gamma_{x,L^{\infty}(A,\tau)}\lc{}R_{\pm i}\rc\rc{}=R_{\pm i}\big(L_{x}^{\phi}\big)$ and $R^{\bpsi}\lc\Gamma_{y,L^{\infty}(A,\tau)}\lc{}R_{\pm i}\rc\rc{}=R_{\pm i}\big(R_{y}^{\bpsi}\big)$.
\end{itemize}
\end{prp}
\begin{proof}
Note $R_{\pm i}$ are resolvents in $\pm i$ \lc{}cf.~Notation \ref{NTN.Resolvents}\rc{}. Let $x,y\in L^{0}(A,\tau)_{+}$. Then $\phi(x),\bpsi(y)\in L^{0}(B,\omega)_{h}$ by Corollary \ref{COR.AF_Cstar_Local_Hom}. Get $\Gamma_{x,L^{\infty}(A,\tau)}\lc{}R_{\pm i}\rc{},\Gamma_{y,L^{\infty}(A,\tau)}\lc{}R_{\pm i}\rc\in L^{\infty}(A,\tau)$ using their bounded measurable functional calculus. Moreover, Proposition \ref{PRP.JFC_Strong_Com} and $2)$ in Lemma \ref{LEM.Wstar_CLRA_FC} imply canonical left-~and right-actions of self-adjoint measurable operators commute strongly. Yet $L^{\phi}=L_{L^{\infty}(B,\omega)}\circ\phi$ and $R^{\bpsi}=R_{L^{\infty}(B,\omega)}\circ\bpsi$. Thus $1)$ follows by Proposition \ref{PRP.AF_Cstar_Bimodule_CLRA_I}. We have

\begin{align}\label{EQ.PRP.AF_Cstar_Bimodule_CLRA_II_1}
\phi\lc\Gamma_{x,L^{\infty}(A,\tau)}\lc{}R_{\pm i}\rc\rc{}=\Gamma_{\phi(x),L^{\infty}(B,\omega)}\lc{}R_{\pm i}\rc{},\ \bpsi\lc\Gamma_{y,L^{\infty}(A,\tau)}\lc{}R_{\pm i}\rc\rc{}=\Gamma_{\bpsi(x),L^{\infty}(B,\omega)}\lc{}R_{\pm i}\rc{}
\end{align}

\noindent by the $^{*}$-homomorphism property. Hence $2)$ follows by Equation \ref{EQ.PRP.AF_Cstar_Bimodule_CLRA_II_1}.
\end{proof}

Definition \ref{DFN.AF_Cstar_Bimodule_Compression} gives compression of AF-$C^{*}$-bimodules by compressing canonical left-~and right-actions. We use the compressibility property in Definition \ref{DFN.JFC_Compression_III}, itself based on Definition \ref{DFN.JFC_Compression_I}, for the pair of normal unital $^{*}$-homomorphisms

\begin{align}\label{EQ.SSEC.NCDS_AF_FC_2}
L^{\phi}:L^{\infty}(A,\tau)\longrightarrow\BII\lc{}L^{2}(B,\omega)\rc{},\ R^{\bpsi}:L^{\infty}(A,\tau)^{\op}\longrightarrow\BII\lc{}L^{2}(B,\omega)\rc{}.
\end{align}

\noindent We give two classes of compression. First, we compress to induced AF-$C^{*}$-bimodules in Corollary \ref{COR.AF_Cstar_Bimodule_Compression_Restriction}. Secondly, we compress with projections in Corollary \ref{COR.AF_Cstar_Bimodule_Projection}.

\begin{dfn}\label{DFN.AF_Cstar_Bimodule_Compression}
Let $N\subset \lc{}L^{\infty}(A,\tau),\tau\rc$ and $V\subset L^{2}(B,\omega)$ be a Hilbert subspace. We say that $(\phi,\bpsi,\gamma)$ is $(N,V)$-compressible, and call $(N,V)$ a compression of $(\phi,\bpsi,\gamma)$, if $\lc{}L^{\phi},R^{\bpsi}\rc$ is $(N,V)$-compressible as per Definition \ref{DFN.JFC_Compression_III} and $\gamma(V)\subset V$.
\end{dfn}

\begin{rem}\label{REM.AF_Cstar_Bimodule_Compression}
Let $N\subset \lc{}L^{\infty}(A,\tau),\tau\rc$ and $V\subset L^{2}(B,\omega)$ be a Hilbert subspace. Note $\pi_{V}^{B}:L^{2}(B,\omega)\longrightarrow V$ is the Hilbert space projection. Then $\lc{}L^{\phi},R^{\bpsi}\rc$ is $(N,V)$-compressible if $L^{\phi}\lc{}L^{\infty}(A,\tau)\rc{},R^{\bpsi}\lc{}L^{\infty}(A,\tau)\rc\subset\BII(V)$ and

\begin{align}\label{EQ.REM.AF_Cstar_Bimodule_Compression_1}
\pi_{V}^{B}=L_{1_{A}}^{\phi}\pi_{V}^{B}=R_{1_{A}}^{\bpsi}\pi_{V}^{B}.    
\end{align}
\end{rem}

\begin{prp}\label{PRP.AF_Cstar_Bimodule_Compression}
Let $N\subset \lc{}L^{\infty}(A,\tau),\tau\rc$ and $V\subset L^{2}(B,\omega)$ be a Hilbert subspace. If $(\phi,\bpsi,\gamma)$ is $(N,V)$-compressible, then

\begin{itemize}
\item[1)] $\phi(N)V\subset V$, $V\bpsi(N)\subset V$ and $\gamma(V)\subset V$,

\item[2)] $V$ equipped with the $(\phi,\bpsi)$-action and $\gamma$ is a symmetric $W^{*}$-bimodule over $N$.
\end{itemize}
\end{prp}
\begin{proof}
Following our discussion in Remark \ref{REM.AF_Cstar_Bimodule_Compression}, we directly verify all claims.
\end{proof}

Following Lemma \ref{LEM.AF_Cstar_Bimodule_Compression_I}, note compressibility as per Definition \ref{DFN.AF_Cstar_Bimodule_Compression} lets us apply Theorem \ref{THM.JFC_Compression}. We use reducing Hilbert subspaces and restriction to Hilbert subspaces given by concrete compression maps \lc{}cf.~Definition \ref{DFN.Compression_Concrete} and Definition \ref{DFN.Reducible}\rc{}. Then Definition \ref{DFN.AF_Cstar_Bimodule_Compression_CLRA} gives compressed canonical left-~and right-actions. Restriction to the bounded case yields compressed induced symmetric $W^{*}$-bimodule actions.

\begin{lem}\label{LEM.AF_Cstar_Bimodule_Compression_I}
Let $N\subset \lc{}L^{\infty}(A,\tau),\tau\rc$ and $V\subset L^{2}(B,\omega)$ be a Hilbert subspace. Let $(\phi,\bpsi,\gamma)$ $(N,V)$-compressible. If $x,y\in L^{0}(N,\tau)_{h}$, then $L_{x}^{\phi},R_{y}^{\bpsi}\in\UBII_{V}\lc{}L^{2}(B,\omega)\rc$ commute strongly and we have 

\begin{align}\label{EQ.PRP.AF_Cstar_Bimodule_Compression_II_1}
L^{\phi}\lc\Gamma_{x,L^{\infty}(A,\tau)}\lc{}R_{\pm i}\rc\rc{}=R_{\pm i}\big(L_{x}^{\phi}\big),\ R^{\bpsi}\lc\Gamma_{y,L^{\infty}(A,\tau)}\lc{}R_{\pm i}\rc\rc{}=R_{\pm i}\big(R_{y}^{\bpsi}\big).   
\end{align}
\end{lem}
\begin{proof}
Let $x,y\in L^{0}(N,\tau)_{h}$. Proposition \ref{PRP.AF_Cstar_Bimodule_CLRA_II} shows all claims except $V$-reducibility of $L_{x}^{\phi}$ and $R_{y}^{\bpsi}$. Set $\phi_{L}:=L^{\phi}\circ L_{L^{\infty}(A,\tau)}^{-1}$ and $\bpsi_{R}:=R^{\bpsi}\circ R_{L^{\infty}(A,\tau)}^{-1}$ on their respective images. Note resolvents are preserved under canonical left-~and right-actions.\par
Arguing as in the proof of Lemma \ref{LEM.RP}, we know mapping $C^{*}$-generators as per Equation \ref{EQ.PRP.AF_Cstar_Bimodule_Compression_II_1} and subsequent closing in $\sigma$-weak operator topology readily yields normal unital $^{*}$-isomorphisms

\begin{align}\label{EQ.PRP.AF_Cstar_Bimodule_Compression_II_2}
\phi_{L}:W^{*}\lc{}L_{x,L^{\infty}(A,\tau)}\rc\longrightarrow W^{*}\big(L_{x}^{\phi}\big),\ \bpsi_{R}:W^{*}\lc{}R_{y,L^{\infty}(A,\tau)}\rc\longrightarrow W^{*}\big(R_{y}^{\bpsi}\big).
\end{align}

\noindent Moreover, said argument for the $^{*}$-homomorphisms in Equation \ref{EQ.PRP.AF_Cstar_Bimodule_Compression_II_2} shows

\begin{align}\label{EQ.PRP.AF_Cstar_Bimodule_Compression_II_3}
\phi_{L}\lc{}E_{L_{x,L^{\infty}(A,\tau)}}(Z)\rc{}=E_{L_{x}^{\phi}}(Z),\ \bpsi_{R}\lc{}E_{R_{y,L^{\infty}(A,\tau)}}(Z)\rc{}=E_{R_{y}^{\bpsi}}(Z).
\end{align}

\noindent for all $Z\in\mathfrak{B}(\mathbb{R})$. Using $2)$ in Lemma \ref{LEM.Wstar_CLRA_FC}, Equation \ref{EQ.PRP.AF_Cstar_Bimodule_Compression_II_3} implies

\begin{align}\label{EQ.PRP.AF_Cstar_Bimodule_Compression_II_4}
\phi\lc{}E_{x,L^{\infty}(A,\tau)}(Z)\rc{}=E_{L_{x}^{\phi}}(Z),\ \bpsi\lc{}E_{L^{\infty}(A,\tau)}(Z)\rc{}=E_{R_{y}^{\bpsi}}(Z)
\end{align}

\noindent in each case. Using $1)$ in Proposition \ref{PRP.AF_Cstar_Bimodule_Compression}, Equation \ref{EQ.PRP.AF_Cstar_Bimodule_Compression_II_4} in turn shows

\begin{align}\label{EQ.PRP.AF_Cstar_Bimodule_Compression_II_5}
\lb{}E_{L_{x}^{\phi}}(Z),\pi_{V}^{B}\rb{}=\lb{}E_{R_{y}^{\bpsi}}(Z),\pi_{V}^{B}\rb{}=0
\end{align}

\noindent for all $Z\in\mathfrak{B}(\mathbb{R})$. Corollary \ref{COR.Compression_Preservation_II} shows Equation \ref{EQ.PRP.AF_Cstar_Bimodule_Compression_II_5} implies $V$-reducibility.
\end{proof}

The $^{*}$-homomorphism property ensures $\phi$ and $\bpsi$ preserve real and imaginary parts \lc{}cf.~Proposition \ref{PRP.Wstar_NCI_IV}\rc{}. Following Proposition \ref{PRP.AF_Cstar_Bimodule_CLRA_I}, note restriction viewed as concrete compression map shows unbounded operators as per Equation \ref{EQ.DFN.AF_Cstar_Bimodule_Compression_CLRA_1} are densely defined and closed \lc{}cf.~Proposition \ref{PRP.Reducible}\rc{}.

\begin{dfn}\label{DFN.AF_Cstar_Bimodule_Compression_CLRA}
Let $N\subset \lc{}L^{\infty}(A,\tau),\tau\rc$ and $V\subset L^{2}(B,\omega)$ be a Hilbert subspace. Let $(\phi,\bpsi,\gamma)$ be $(N,V)$-compressible. For all $x,y\in L^{0}(A,\tau)$, set

\begin{align}\label{EQ.DFN.AF_Cstar_Bimodule_Compression_CLRA_1}
L_{x,N}^{\phi}:=\restr{0.875}{L_{x}^{\phi}}{V},\ R_{y,N}^{\bpsi}:=\restr{0.875}{R_{y}^{\bpsi}}{V}
\end{align}
\end{dfn}

\begin{ntn}\label{NTN.AF_Cstar_Bimodule_Compression}
We suppress $\phi$ and $\bpsi$ in Definition \ref{DFN.AF_Cstar_Bimodule_Compression_CLRA} if $\phi=\bpsi=\id_{A}$. We further suppress $N$ if $N=L^{\infty}(A,\tau)$. In particular, $L_{\phi(x)}$ and $R_{\bpsi(y)}$ denote evaluated canonical left-~and right-actions of $L^{0}(B,\omega)$ on $L^{2}(B,\omega)$.
\end{ntn}

\begin{lem}\label{LEM.AF_Cstar_Bimodule_Compression_II}
Let $N_{A}\subset \lc{}L^{\infty}(A,\tau),\tau\rc$ and $N_{B}\subset \lc{}L^{\infty}(B,\omega),\omega\rc$. If

\begin{itemize}
\item[1)] $\phi\lc{}N_{A}\rc{},\bpsi\lc{}N_{A}\rc\subset N_{B}$ and $\phi(1_{N_{A}})=\bpsi(1_{N_{A}})=1_{N_{B}}$,

\item[2)] $\gamma\lc{}N_{B}\cap L^{2}(B,\omega)\rc\subset N_{B}\cap L^{2}(B,\omega)$,
\end{itemize}

\noindent then $(\phi,\bpsi,\gamma)$ is $\lc{}N_{A},L^{2}\lc{}N_{B},\omega\rc\rc$-compressible.
\end{lem}
\begin{proof}
Following our discussion in Remark \ref{REM.AF_Cstar_Bimodule_Compression}, we directly verify all claims.
\end{proof}

\begin{cor}\label{COR.AF_Cstar_Bimodule_Compression_Restriction}
For all $j\in\mathbb{N}$, $(\phi,\bpsi,\gamma)$ is $(A_{j},B_{j})$-compressible.
\end{cor}
\begin{proof}
Let $j\in\mathbb{N}$. Apply Lemma \ref{LEM.AF_Cstar_Bimodule_Compression_II} to $N_{A}=A_{j}$ and $N_{B}=B_{j}$.
\end{proof}

Let $p\in L^{\infty}(A,\tau)$ be a projection. We know $L^{\infty}(A,\tau)[p]\subset \lc{}L^{\infty}(A,\tau),\tau\rc$. Lemma \ref{LEM.Cstar_Trace_Abstract_Projection} shows $(A[p],\tau)$ is a tracial $C^{*}$-algebra in $L^{\infty}(A,\tau)[p]$. Note $L^{\infty}(A[p],\tau)=pL^{\infty}(A,\tau)p$.

\begin{dfn}\label{DFN.AF_Cstar_Bimodule_Projection}
Let $p\in L^{\infty}(A,\tau)$ be a projection. Set $L^{2}(B[p],\omega):=pL^{2}(B,\omega)p$. For all $u\in L^{2}(B,\omega)$, further set

\begin{align}\label{EQ.DFN.AF_Cstar_Bimodule_Projection_1}
\pi_{p}(u):=pup,\ \pi_{p}^{\perp}(u):=pup^{\perp}+p^{\perp}up+p^{\perp}up^{\perp}.
\end{align}
\end{dfn}

\begin{cor}\label{COR.AF_Cstar_Bimodule_Projection}
For all projections $p\in L^{\infty}(A,\tau)$, we have

\begin{itemize}
\item[1)] $L^{2}(B[p],\omega)\subset L^{2}(B,\omega)$ is a Hilbert subspace and $\pi_{L^{2}(B[p],\omega)}^{B}=\pi_{p}$,

\item[2)] $(\phi,\bpsi,\gamma)$ is $\lc{}L^{\infty}(A[p],\tau),L^{2}(B[p],\omega)\rc$-compressible.
\end{itemize}
\end{cor}
\begin{proof}
Apply Lemma \ref{LEM.Cstar_Trace_Abstract_Projection}, $2)$ in Proposition \ref{PRP.Wstar_Trace_NCE_II} and Equation \ref{EQ.DFN.AF_Cstar_Bimodule_Projection_1}.
\end{proof}


\subsubsection*{Functional calculus}

Following Lemma \ref{LEM.AF_Cstar_Bimodule_Compression_I}, we apply Theorem \ref{THM.JFC_Compression} to get compressed pulled-back joint functional calculus of extended AF-$C^{*}$-bimodule actions \lc{}cf.~Definition \ref{DFN.JFC_Compression_IV}\rc{}. For this, we compress joint functional calculus pulled-back along canonical left-~and right-actions of AF-$C^{*}$-bimodules. We use joint functional calculus of strongly commuting self-adjoint unbounded operators \lc{}cf.~Definition \ref{DFN.JFC_Unbd}\rc{}, as well as bounded measurable joint functional calculus of self-adjoint measurable operators \lc{}cf.~Definition \ref{DFN.Wstar_CLRA_FC_VI}\rc{}. For details on the former, we refer to Subsection \ref{SSEC.A_Fnd_FC}. For details on the latter, we refer to Subsection \ref{SSEC.B_SMO_CLRA}.\par


\pagebreak


Let $(A,\tau)$ and $(B,\omega)$ be tracial AF-$C^{*}$-algebras. Let $(\phi,\bpsi,\gamma)$ be an AF-$A$-bimodule structure on $B$. Let $N\subset \lc{}L^{\infty}(A,\tau),\tau\rc$ and $V\subset L^{2}(B,\omega)$ be a Hilbert subspace. Assume $(\phi,\bpsi,\gamma)$ is $(N,V)$-compressible. Let $x,y\in L^{0}(A,\tau)_{h}$. Lemma \ref{LEM.AF_Cstar_Bimodule_Compression_I} shows Theorem \ref{THM.JFC_Compression} applies to $T=L_{x,N}^{\phi}$ and $S=R_{y,N}^{\bpsi}$ using $\lc{}L^{\phi},R^{\bpsi}\rc$ as $(N,V)$-compressible pair. Following $1)$ in Definition \ref{DFN.JFC_Compression_IV}, we have bounded measurable joint functional calculus

\begin{align}\label{EQ.SSEC.NCDS_AF_FC_3}
\Gamma_{x,y,N}^{L^{\phi},R^{\bpsi}}:L^{\infty}\lc\specN x\times y,dE_{x,y,N}\rc\longrightarrow\BII(V)
\end{align}

\noindent of $x\otimes y$ in $N\otimes N^{\op}$ under $L^{\phi}\otimes R^{\bpsi}$. Following $2)$ and $3)$ in Definition \ref{DFN.JFC_Compression_IV}, we have joint functional calculus

\begin{align}\label{EQ.SSEC.NCDS_AF_FC_4}
\Gamma_{x,y,N}^{L^{\phi},R^{\bpsi}}:\mathcal{S}_{V}\lc{}E_{x,y,N}\rc\longrightarrow\UBII\lc{}V\rc_{h}
\end{align}

\noindent of $x\otimes y$ in $N\otimes N^{\op}$ under $L^{\phi}\otimes R^{\bpsi}$ \lc{}cf.~Corollary \ref{COR.JFC_Compression_I}\rc{}.

\begin{rem}
In the setting of $5)$ in Definition \ref{DFN.AF_Cstar_Bimodule}, we have $\Gamma_{x,y,N}^{L^{\phi},R^{\bpsi}}=\Gamma_{x,y,N}$.
\end{rem}

Let $H$ be a Hilbert space. If $V\subset H$ is a Hilbert subspace and $\pi_{V}:H\longrightarrow V$ its Hilbert space projection, then get positivity-preserving canonical inclusion $\UBII(V)\subset\UBII(H)$ by mapping $T\mapsto \comV T=\pi_{V}T\pi_{V}$. This yields $\BII(V)\oplus\BII(V^{\perp})\subset\BII(H)$ \lc{}cf.~Equation \ref{EQ.SSEC.A_Maps_Compression_2}\rc{}.

\begin{lem}\label{LEM.AF_NCD_FC_I}
Let $(\phi,\bpsi,\gamma)$ be $(N,V)$-compressible. We consider $L^{2}(B,\omega)=V\oplus V^{\perp}$.

\begin{itemize}
\item[1)] For all $x\in L^{0}(N,\tau)_{h}$, $L_{x}^{\phi}=L_{x,N}^{\phi}+L_{x}^{\phi}\lc{}I-\pi_{V}^{B}\rc$ and $R_{x}^{\bpsi}=R_{x,N}^{\bpsi}+R_{x}^{\bpsi}\lc{}I-\pi_{V}^{B}\rc$.

\item[2)] Let $x\in L^{0}(N,\tau)_{h}$. If $g\in L^{\infty}\lc\spec_{L^{\infty}(A,\tau)} x\times y,dE_{x,y,L^{\infty}(A,\tau)}\rc$, then

\begin{itemize}
\item[2.1)] $g\in L^{\infty}\lc\specN x\times y,dE_{x,y,N}\rc$ by restricting to $\specN x\times y\subset\spec_{L^{\infty}(A,\tau)}x\times y$, \phantom{\bigg)}

\item[2.2)] $g\lc{}L_{x}^{\phi},R_{y}^{\bpsi}\rc\in\BII\lc{}L^{2}(B,\omega)\rc\cap\UBII_{V}\lc{}L^{2}(B,\omega)\rc$ and $\restr{0.925}{g\lc{}L_{x}^{\phi},R_{y}^{\bpsi}\rc}{V}=g\lc{}L_{x,N}^{\phi},R_{y,N}^{\bpsi}\rc$, \phantom{\bigg)}

\item[2.3)] $g\lc{}L_{x}^{\phi},R_{y}^{\bpsi}\rc{}=g\lc{}L_{x,N}^{\phi},R_{y,N}^{\bpsi}\rc\oplus \restr{0.925}{g\lc{}L_{x}^{\phi},R_{y}^{\bpsi}\rc}{V^{\perp}}\in\BII(V)\oplus\BII(V^{\perp})$. \phantom{\bigg)}
\end{itemize}

\begin{reapply}
\end{reapply}

\item[3)] If $x,y\in L^{0}(N,\tau)_{+}$, $\alpha,\beta\geq 0$ and $g\in C_{b}\lc{}[0,\infty)\times [0,\infty)\rc$, then

\begin{align}\label{EQ.LEM.AF_NCD_FC_I_1}
\restr{0.875}{g\bigg(L_{x+\alpha 1_{N}^{\perp}}^{\phi},R_{y+\beta 1_{N}^{\perp}}^{\bpsi}\bigg)}{V}=g\lc{}L_{x,N}^{\phi},R_{y,N}^{\bpsi}\rc{}.
\end{align}

\begin{reapply}
\end{reapply}

\end{itemize}
\end{lem}
\begin{proof}
Let $x,y\in L^{0}(N,\tau)_{h}$. Lemma \ref{LEM.AF_Cstar_Bimodule_Compression_I} shows $L_{x}^{\phi},R_{y}^{\bpsi}\in\UBII_{V}\lc{}L^{2}(B,\omega)\rc$, i.e.~each is a $V$-reducible self-adjoint unbounded operator. Get $1)$ by $1.3)$ in Proposition \ref{PRP.Reducible}. Note using abstract and concrete spectral measures yields identical commutative $L^{\infty}$-spaces in the uncompressed, resp.~compressed case.\par


\pagebreak


We therefore have

\begin{align}\label{EQ.LEM.AF_NCD_FC_I_2}
\Gamma_{L_{x}^{\phi},R_{y}^{\bpsi}}(g)=g\lc{}L_{x}^{\phi},R_{y}^{\bpsi}\rc{},\ \Gamma_{L_{x,N}^{\phi},R_{y,N}^{\bpsi}}(g)=g\lc{}L_{x,N}^{\phi},R_{y,N}^{\bpsi}\rc{}=g\lc\restr{0.875}{L_{x}^{\phi}}{V},\restr{0.875}{R_{y}^{\bpsi}}{V}\rc{}
\end{align}

\noindent in each case. Following Lemma \ref{LEM.AF_Cstar_Bimodule_Compression_I}, we apply Theorem \ref{THM.JFC_Compression} as discussed above. Spectra restrict as claimed and have

\begin{align}\label{EQ.LEM.AF_NCD_FC_I_3}
\restr{0.925}{g\lc{}L_{x}^{\phi},R_{y}^{\bpsi}\rc}{V}=g\lc\restr{0.875}{L_{x}^{\phi}}{V},\restr{0.875}{R_{y}^{\bpsi}}{V}\rc{}.
\end{align}

\noindent Equation \ref{EQ.LEM.AF_NCD_FC_I_2} and Equation \ref{EQ.LEM.AF_NCD_FC_I_3} imply $2.1)$ and $2.2)$ at once. Using $2.1)$, get $2.3)$ by $1.3)$ in Proposition \ref{PRP.Reducible}. We have direct sum by boundedness. Altogether, get $2)$.\par
We show $3)$. Let $x,y\in L^{0}(N,\tau)_{+}$, $\alpha,\beta\geq 0$ and $g\in C_{b}\lc{}[0,\infty)\times [0,\infty)\rc$. Theorem \ref{THM.JFC_Compression} implies Equation \ref{EQ.LEM.AF_NCD_FC_I_1} is

\begin{align}\label{EQ.LEM.AF_NCD_FC_I_4}
\lc{}L^{\phi}\otimes_{V} R^{\bpsi}\rc\lc\comunitunit\lc\Gamma_{x+\alpha 1_{N}^{\perp},y+\beta 1_{N}^{\perp},L^{\infty}(A,\tau)}(g)\rc\rc{}=\lc{}L^{\phi}\otimes_{V} R^{\bpsi}\rc\lc\Gamma_{x,y,N}(g)\rc{}.
\end{align}

\noindent Applying $L^{\phi}\otimes_{V} R^{\bpsi}$ to Equation \ref{EQ.COR.JFC_Compression_III_1} in Corollary \ref{COR.JFC_Compression_III} yields Equation \ref{EQ.LEM.AF_NCD_FC_I_4}.
\end{proof}

\begin{ntn}\label{NTN.AF_Cstar_Bimodule_Projection}
Assume $(N,V)=\lc{}L^{\infty}(A[p],\tau),L^{2}(B[p],\omega)\rc$ for projection $p\in L^{\infty}(A,\tau)$. For all $x,y\in L^{0}(A,\tau)_{h}$, we write

\begin{itemize}
\item[1)] $L_{x,p}^{\phi}:=L_{x,L^{\infty}(A[p],\tau)}^{\phi}$ and $R_{y,p}^{\bpsi}:=R_{y,L^{\infty}(A[p],\tau)}^{\bpsi}$,

\item[2)] $\Gamma_{x,y,p}^{\phi,\bpsi}:=\Gamma_{x,y,L^{\infty}(A[p],\tau)}^{L^{\phi},R^{\bpsi}}$ and $\SIIp\lc{}E_{x,y}\rc:=\mathcal{S}_{L^{2}(B[p],\omega)}\lc{}E_{x,y,L^{\infty}(A[p],\tau)}\rc$.
\end{itemize}

\noindent We suppress $\phi$ and $\bpsi$ if $\phi=\bpsi=\id_{A}$. We further suppress $p$ if $p=1_{A}$.
\end{ntn}

\begin{lem}\label{LEM.AF_NCD_FC_II}
Assume $(N,V)=\lc{}L^{\infty}(A[p],\tau),L^{2}(B[p],\omega)\rc$ for projection $p\in L^{\infty}(A,\tau)$. If $x,y\in L^{0}(A[p],\tau)_{+}$ and $g\in C_{b}\lc{}[0,\infty)\times [0,\infty)\rc$, then

\begin{itemize}
\item[1)] $g\lc{}L_{x}^{\phi},R_{y}^{\bpsi}\rc\pi_{p}=g\lc{}L_{x,p}^{\phi},R_{y,p}^{\bpsi}\rc\pi_{p}$,

\item[2)] $g\lc{}L_{x}^{\phi},R_{y}^{\bpsi}\rc\pi_{p}^{\perp}=g\lc{}L_{x}^{\phi},0\rc{}L_{p}^{\phi}\mathrlap{\phantom{R}^{\bpsi}}R_{p^{\perp}}+g\lc{}0,R_{y}^{\bpsi}\rc\mathrlap{\phantom{L}^{\phi}}L_{p^{\perp}}R_{p}^{\bpsi}+g\lc{}0,0\rc\pi_{p^{\perp}}$.
\end{itemize}
\end{lem}
\begin{proof}
For details on the tensor product of normal unital $^{*}$-homomorphisms, we refer to Corollary \ref{COR.Wstar_TP}. By definition, we have

\begin{align}\label{EQ.LEM.AF_NCD_FC_II_1}
\Gamma_{x,y,L^{\infty}(A,\tau)}^{L^{\phi},R^{\bpsi}}=\lc{}L^{\phi}\otimes R^{\bpsi}\rc\circ\Gamma_{x,y,L^{\infty}(A,\tau)}.
\end{align}


\pagebreak


Following Notation \ref{NTN.AF_Cstar_Bimodule_Projection}, Equation \ref{EQ.DFN.AF_Cstar_Bimodule_Projection_1} rewrites as

\begin{align}\label{EQ.LEM.AF_NCD_FC_II_2}
\pi_{p}=\pi_{L^{2}(B[p],\omega)}^{B}=L_{p}^{\phi}R_{p}^{\bpsi},\ \pi_{p}^{\perp}=L_{p}^{\phi}\mathrlap{\phantom{R}^{\bpsi}}R_{p^{\perp}}+\mathrlap{\phantom{L}^{\phi}}L_{p^{\perp}}R_{p}^{\bpsi}+\mathrlap{\phantom{L}^{\phi}}L_{p^{\perp}}\mathrlap{\phantom{R}^{\bpsi}}R_{p^{\perp}}.
\end{align}

\noindent Using $\pi_{p}$ itself as our Hilbert space projection, get $1)$ by $2)$ in Lemma \ref{LEM.AF_NCD_FC_I} and $1.1)$ in Proposition \ref{PRP.Reducible}. We show $2)$ by arguing as in the proof of Corollary \ref{COR.JFC_Compression_III}. We assume $p,p^{\perp}\in W_{L^{\infty}(A,\tau)}^{*}(x)\cap W_{L^{\infty}(A,\tau)}^{*}(y)$ without loss of generality. Since $W_{L^{\infty}(A,\tau)}^{*}(x,y)=W_{L^{\infty}(A,\tau)}^{*}(x)\otimes W_{L^{\infty}(A,\tau)}^{*}(y)$ \lc{}cf.~$2)$ in Definition \ref{DFN.Wstar_CLRA_FC_IV}\rc{}, Equation \ref{EQ.LEM.AF_NCD_FC_II_1} and Equation \ref{EQ.LEM.AF_NCD_FC_II_2} let us calculate

\begin{align}\label{EQ.LEM.AF_NCD_FC_II_3}
\lc{}L^{\phi}\otimes R^{\bpsi}\rc^{-1}\lc\Gamma_{x,y,L^{\infty}(A,\tau)}^{\phi,\bpsi}(g)\pi_{p}^{\perp}\rc{}=\Gamma_{x,y,L^{\infty}(A,\tau)}(g)\lc{}p\otimes p^{\perp}+p^{\perp}\otimes p+p^{\perp}\otimes p^{\perp}\rc{}.
\end{align}

\noindent We write each summand in Equation \ref{EQ.LEM.AF_NCD_FC_II_3} as element in $W_{L^{\infty}(A,\tau)}^{*}(x,y)$. Theorem \ref{THM.JFC_Compression} then implies

\vspace{-0.375cm}
\begin{align}\label{EQ.LEM.AF_NCD_FC_II_4}
\Gamma_{x,y,L^{\infty}(A,\tau)}(g)\lc{}p\otimes p^{\perp}\rc{} =& \ \Gamma_{x,0,L^{\infty}(A,\tau)}(g)\lc{}p\otimes p^{\perp}\rc{}, \phantom{\vstretch{0.9175}{\bigg)}} \\
\Gamma_{x,y,L^{\infty}(A,\tau)}(g)\lc{}p^{\perp}\otimes p\rc{} =& \ \Gamma_{0,y,L^{\infty}(A,\tau)}(g)\lc{}p^{\perp}\otimes p\rc{}, \phantom{\vstretch{0.9175}{\bigg)}} \\
\Gamma_{x,y,L^{\infty}(A,\tau)}(g)\lc{}p^{\perp}\otimes p^{\perp}\rc{} =& \ \Gamma_{0,0,L^{\infty}(A,\tau)}(g)\lc{}p^{\perp}\otimes p^{\perp}\rc{}. \phantom{\vstretch{0.9175}{\bigg)}}
\end{align}

\noindent Upon applying $L^{\phi}\otimes R^{\bpsi}$ to Equation \ref{EQ.LEM.AF_NCD_FC_II_3}, the above equations show $2)$ at once.
\end{proof}


\section{Noncommutative division operators}\label{SEC.NCDS_NCD}

Noncommutative division operators generalise division by densities in the classical case \cite{ART.Dol_Naz_Sav.2009.Generalised_OT}. In the tracial finite-dimensional cases of \cite{ART.Car_Maa.2014.Quantum_OT_I}\cite{ART.Car_Maa.2017.Quantum_OT_II}\cite{ART.Car_Maa.2020.Quantum_OT_III}, they determine, and are in turn determined by, quasi-entropies \cite{ART.Hia_Pet.2012.Quasi_Entropy_I}\cite{ART.Hia_Pet.2013.Quasi_Entropy_II} used to define energy functionals. Note\linebreak quasi-entropies generalise quantum $f$-divergences \cite{ART.Hia.2018.QFD_I}\cite{ART.Hia.2019.QFD_II}, a class of dissimilarity measures for information encoded in states of quantum systems \cite{BK.Nie_Chu.2000.Quantum_Computation_Information}\cite{ART.Kra.1971.State_Changes}. Applying the Kato-Robinson theorem \cite{BK.deOli.2009.OpAlg_Quantum_Dynamics} lets us extend the approach in \cite{ART.Car_Maa.2020.Quantum_OT_III} to AF-$C^{*}$-bimodules.\par
Noncommutative division operators represent closed positive unbounded quadratic forms determined by quasi-entropies. Quasi-entropies are non-negative, jointly convex and $w^{*}$-l.s.c.~functionals on Banach dual spaces of AF-$C^{*}$-bimodules. The Kato-Robinson theorem shows noncommutative division operators are strong resolvent limits of, upon suitable evaluation for each, perturbed inverses of operator means \cite{ART.And_Kub.1979.Operator_Means} as perturbation tends to zero. Such perturbed inverses are expressed using compressed pulled-back joint functional calculus of extended AF-$C^{*}$-bimodule actions. We recover noncommutative division operators of positive measurable operators if and only if the strong resolvent limit is likewise expressed using compressed pulled-back joint functional calculus.

\medskip

\noindent\textbf{Structure.} In Subsection \ref{SSEC.NCDS_NCD_QE}, we discuss operator means, noncommutative division operators of positive measurable operators and quasi-entropies for AF-$C^{*}$-bimodules. In Subsection \ref{SSEC.NCDS_NCD_Operators}, we represent closed positive unbounded quadratic forms determined by quasi-entropies using noncommutative division operators.


\subsection[Quasi-entropies for AF-$C^{*}$-bimodules]{Quasi-entropies for AF-$\mathbf{C}^{*}$-bimodules}\label{SSEC.NCDS_NCD_QE}

We define noncommutative division operators of positive measurable operators, as well as perturbed ones. Following Lemma \ref{LEM.AF_NCD_FC_I} and Lemma \ref{LEM.AF_NCD_FC_II}, we have control as per Lemma \ref{LEM.NCD_Operator_Compressed_PMO_I}. We define quasi-entropies in the finite-dimensional setting by letting perturbation tend to zero upon applying perturbed noncommutative division operators of positive measurable operators. Using monotonicity under restriction maps, we extend quasi-entropies to AF-$C^{*}$-bimodules. Theorem \ref{THM.QE_AF} collects fundamental properties.


\subsubsection*{Noncommutative division operators of positive measurable operators}

\hspace{-0.15cm} Let $(A,\tau)$ and $(B,\omega)$ be tracial AF-$C^{*}$-algebras. Let $(\phi,\bpsi,\gamma)$ be an AF-$A$-bimodule structure on $B$. Let $N\subset \lc{}L^{\infty}(A,\tau),\tau\rc$ and $V\subset L^{2}(B,\omega)$ be a Hilbert subspace. Let $f$ be representing function of an operator mean as per $2)$ in Definition \ref{DFN.OM_Representing_Fct} and $\theta\in [0,1]$.

\begin{dfn}\label{DFN.OM_Representing_Fct}
Let $f:(0,\infty)\longrightarrow (0,\infty)$.

\begin{itemize}
\item[1)] We call $f$ symmetric if $f(t)=f(t^{-1})$ for all $t>0$. 

\item[2)] We call $f$ representing function of an operator mean, or representing function if it is operator monotone and $f(1)=1$. We define its mean $m_{f}:(0,\infty)\times (0,\infty)\longrightarrow (0,\infty)$ by setting $m_{f}(t,s):=f(ts^{-1})s$ for all $t,s>0$. For all $\varepsilon>0$, we furthermore define its mean $m_{f,\varepsilon}:[0,\infty)\longrightarrow (0,\infty)$ perturbed with $\varepsilon$ by setting $m_{f,\varepsilon}(t,s):=m_{f}\lc{}t+\varepsilon,s+\varepsilon\rc$ for all $t,s\geq 0$.

\item[3)] Let $\mathcal{A}$ be a unital $^{*}$-algebra equipped with partial order generated by positive elements. Set $\mathcal{A}_{>0}:=\lset{}x\in A_{+}\ \vset\ \exists\varepsilon>0:\ x\geq\varepsilon 1_{\mathcal{A}} \rset$. For all $x\in \mathcal{A}$, we say that $x>0$ in $\mathcal{A}$ if $x\in \mathcal{A}_{>0}$.
\end{itemize}
\end{dfn}

\begin{rem}
If $f$ is symmetric, then $m_{f}(t,s)=m_{f}(s,t)$ for all $t,s>0$. If $f$ is a representing function, then given separable Hilbert space $H$ and letting $m_{f}\lc{}T,S\rc$ for all commuting $T,S>0$ in $\BII(H)$ defines operator mean following Kubo and Ando \cite{ART.And_Kub.1979.Operator_Means}.
\end{rem}

\begin{prp}\label{PRP.OM_Representing_Fct}
For all $t_{1}\geq t_{0}>0$ and $s_{1}\geq s_{0}>0$, get $m_{f}(t_{1},s_{1})\geq m_{f}(t_{0},s_{0})$. There exists unique continuous extension of $m_{f}$ to $[0,\infty)\times [0,\infty)$. 
\end{prp}
\begin{proof}
Let $\mathbb{C}\cong\langle I\rangle_{\mathbb{C}}\subset\BII(H)$ for a separable Hilbert space $H$. For all $t,s>0$, get $m_{f}(t,s)=f\lc{}tI\cdot s^{-1}I\rc\cdot sI$. Operator means are connections by Theorem 3.2 in \cite{ART.And_Kub.1979.Operator_Means}. We see our first claim follows by (I), and our second one by \lc{}III\rc{} on p.206 in \cite{ART.And_Kub.1979.Operator_Means}.
\end{proof}

\begin{rem}\label{REM.OM_Representing_Fct}
For all $\varepsilon>0$, get $m_{f,\varepsilon}^{-1}\in C_{b}\lc[\varepsilon,\infty)\times [\varepsilon,\infty)\rc$ by Proposition \ref{PRP.OM_Representing_Fct}.
\end{rem}

Definition \ref{DFN.NCD_Operator_Compressed_PMO} uses joint functional calculus to give noncommutative multiplication and division operators of positive measurable operators. Proposition \ref{PRP.NCD_Operator_Compressed_PMO_I} ensures $2)$ and $3)$ in Definition \ref{DFN.NCD_Operator_Compressed_PMO} are well-defined.\par


\pagebreak


\begin{prp}\label{PRP.NCD_Operator_Compressed_PMO_I}
Let $(\phi,\bpsi,\gamma)$ be $(N,V)$-compressible.

\begin{itemize}
\item[1)] If $x,y\in L^{0}(N,\tau)_{+}$ s.t.~$m_{f}^{-1}\in\SII\lc{}E_{x,y,N}\rc$, then $m_{f}^{-\theta}\in\mathcal{S}_{V}\lc{}E_{x,y,N}\rc$.

\item[2)] If $x>0$ in $L^{0}(N,\tau)$, then there exists $\varepsilon>0$ s.t.~$\specN x\subset [\varepsilon,\infty)$.
\end{itemize}
\end{prp}
\begin{proof}
Let $x,y\in L^{0}(N,\tau)_{+}$. Note $I_{V}\in\BII(V)$ is the unit. By functional calculus, get

\begin{align}\label{EQ.PRP.NCD_Operator_Compressed_PMO_I_1}
m_{f}^{-\theta}\lc{}L_{x,N}^{\phi}+\varepsilon I_{V},R_{y,N}^{\bpsi}+\varepsilon I_{V}\rc{}=m_{f,\varepsilon}^{-\theta}\lc{}L_{x,N}^{\phi},R_{y,N}^{\bpsi}\rc{}.
\end{align}

\noindent Let $\lset\varepsilon_{n}\rset_{n\in\mathbb{N}}\subset (0,\infty)$ be descending sequence converging to zero. Then Proposition 10.1.8 in \cite{BK.deOli.2009.OpAlg_Quantum_Dynamics} implies

\begin{align}\label{EQ.PRP.NCD_Operator_Compressed_PMO_I_2}
L_{x,N}^{\phi}=\sr\textrm{-}\lim_{n\in\mathbb{N}}\hspace{0.025cm} L_{x,N}^{\phi}+\varepsilon_{n}I_{V},\ R_{y,N}^{\bpsi}=\sr\textrm{-}\lim_{n\in\mathbb{N}}\hspace{0.025cm} R_{y,N}^{\bpsi}+\varepsilon_{n}I_{V}.
\end{align}

\noindent We show $1)$. Assume $m_{f}^{-1}\in\SII\lc{}E_{x,y,N}\rc$. Note, by definition, $1)$ holds if $m_{f}^{-\theta}$ satisfies $1)$ and $2)$ in Corollary \ref{COR.JFC_Compression_I}. Get $1)$ in the corollary by Remark \ref{REM.OM_Representing_Fct}. In order to get $2)$ in the corollary, we calculate strong convergence of resolvents. Using Equation \ref{EQ.PRP.NCD_Operator_Compressed_PMO_I_1} and Equation \ref{EQ.PRP.NCD_Operator_Compressed_PMO_I_2}, we apply Lemma \ref{LEM.FC_SR} in the one-variable case to get

\begin{align*}
R_{\pm i}\lc{}m_{f}^{-\theta}\lc{}L_{x,N}^{\phi},R_{y,N}^{\bpsi}\rc\rc{} & = \s\textrm{-}\lim_{\varepsilon\downarrow 0}\hspace{0.025cm} R_{\pm i}\lc{}m_{f}^{-\theta}\lc{}L_{x,N}^{\phi}+\varepsilon I_{V},R_{y,N}^{\bpsi}+\varepsilon I_{V}\rc\rc{} \phantom{\Bigg)} \\
& = \s\textrm{-}\lim_{\varepsilon\downarrow 0}\hspace{0.025cm} R_{\pm i}\lc{}m_{f,\varepsilon}^{-\theta}\lc{}L_{x,N}^{\phi},R_{y,N}^{\bpsi}\rc\rc{}. \phantom{\Bigg)}
\end{align*}

\noindent Get $1)$. Using $2)$ in Lemma \ref{LEM.Wstar_CLRA_FC} and Corollary \ref{COR.Wstar_CLRA_III}, we directly verify $2)$.
\end{proof}

\begin{dfn}\label{DFN.NCD_Operator_Compressed_PMO}
Let $(\phi,\bpsi,\gamma)$ be $(N,V)$-compressible. For all $x,y\in L^{0}(N,\tau)_{+}$, we define

\begin{itemize}
\item[1)] $\mathcal{M}_{x,y,N}:=m_{f}\lc{}L_{x,N}^{\phi},R_{y,N}^{\bpsi}\rc$ and $\mathcal{M}_{x,N}:=\mathcal{M}_{x,x,N}$,

\item[2)] $\mathcal{D}_{x,y,N}:=\mathcal{M}_{x,y,N}^{-1}$ and $\mathcal{D}_{x,N}:=\mathcal{D}_{x,x,N}$ if $m_{f}^{-1}\in\SII\lc{}E_{x,y,N}\rc$,

\item[3)] $\mathcal{D}_{x,y,N,\varepsilon}:=\mathcal{D}_{x+\varepsilon 1_{N},y+\varepsilon 1_{N},N}$ for all $\varepsilon>0$.
\end{itemize}
\end{dfn}

\begin{ntn}\label{NTN.NCD_Operator_Compressed_PMO}
We suppress $N$ in Definition \ref{DFN.NCD_Operator_Compressed_PMO} if $N=L^{\infty}(A,\tau)$.
\end{ntn}

\begin{rem}
All unbounded operators in Definition \ref{DFN.NCD_Operator_Compressed_PMO} are positive. If $x,y>0$ in $L^{0}(N,\tau)$, then $\mathcal{D}_{x,y,N}\in\BII\lc{}V\rc_{+}$ by $2)$ in Proposition \ref{PRP.NCD_Operator_Compressed_PMO_I}. If further $x,y\in L^{\infty}(A,\tau)$, then $\mathcal{D}_{x,y,N}>0$ in $\BII(V)$ by construction.
\end{rem}


\pagebreak


\begin{ntn}
Assume $(N,V)=\lc{}L^{\infty}(A[p],\tau),L^{2}(B[p],\omega)\rc$ for projection $p\in L^{\infty}(A,\tau)$. For all $x,y\in L^{0}(A[p],\tau)_{+}$, we write

\begin{itemize}
\item[1)] $\mathcal{M}_{x,y,p}:=\mathcal{M}_{x,y,L^{\infty}(A[p],\tau)}$ and $\mathcal{M}_{x,p}:=\mathcal{M}_{x,x,p}$,

\item[2)] $\mathcal{D}_{x,y,p}:=\mathcal{D}_{x,y,L^{\infty}(A[p],\tau)}$ and $\mathcal{D}_{x,p}:=\mathcal{D}_{x,x,L^{\infty}(A[p],\tau)}$ if $m_{f}^{-1}\in\SIIp\lc{}E_{x,y}\rc$,
\item[3)] $\mathcal{D}_{x,y,p,\varepsilon}:=\mathcal{D}_{x,y,L^{\infty}(A[p],\tau),\varepsilon}$ for all $\varepsilon>0$.
\end{itemize}

\noindent We suppress $p$ if $p=1_{A}$.
\end{ntn}

\begin{prp}\label{PRP.NCD_Operator_Compressed_PMO_II}
Let $(\phi,\bpsi,\gamma)$ be $(N,V)$-compressible. For all $x,y\in L^{0}(N,\tau)_{+}$, we have

\begin{itemize}
\item[1)] $\mathcal{D}_{x,y,N,\varepsilon}^{\theta}=m_{f}^{-\theta}\lc{}L_{x,N}^{\phi}+\varepsilon I_{V},R_{y,N}^{\bpsi}+\varepsilon I_{V}\rc{}=m_{f,\varepsilon}^{-\theta}\lc{}L_{x,N}^{\phi},R_{y,N}^{\bpsi}\rc$ for all $\varepsilon>0$,

\item[2)] $\mathcal{D}_{x,y,N,\varepsilon_{1}}^{\theta}\leq\mathcal{D}_{x,y,N,\varepsilon_{0}}^{\theta}$ in $\BII(V)$ for all $\varepsilon_{1}\geq\varepsilon_{0}>0$ in $\mathbb{R}$,

\item[3)] $\mathcal{D}_{x,y,N}^{\theta}=\sr$-$\lim_{\varepsilon\downarrow 0}\mathcal{D}_{x,y,N,\varepsilon}^{\theta}$ if $m_{f}^{-1}\in\SII\lc{}E_{x,y,N}\rc$.
\end{itemize}
\end{prp}
\begin{proof}
Since $I_{V}=L_{1_{N},N}^{\phi}=R_{1_{N},N}^{\bpsi}$ by unitality as per $2)$ in Lemma \ref{LEM.AF_Cstar_Local_Hom}, Equation \ref{EQ.PRP.NCD_Operator_Compressed_PMO_I_1} shows $1)$. Bounded measurable joint functional calculus is positivity-preserving since it is a normal unital $^{*}$-homomorphism \lc{}cf.~$1)$ in Proposition \ref{PRP.JFC_Bd}\rc{}. Thus $2)$ follows from $1)$ and Proposition \ref{PRP.OM_Representing_Fct}. Proposition \ref{PRP.NCD_Operator_Compressed_PMO_II} shows $3)$ follows from Corollary \ref{COR.JFC_Compression_I}.
\end{proof}

\begin{lem}\label{LEM.AF_NCD_FC_III}
Let $(\phi,\bpsi,\gamma)$ be $(N,V)$-compressible. Let $x\in N_{h}$ and $g\in C_{b}\lc\mathbb{R}\times\mathbb{R}\rc$. If $K\subset\mathbb{R}$ is compact s.t.~$\specN x\subset K$ and $g(t,s)=g(s,t)$ for all $t,s\in K$, then 

\begin{align}\label{EQ.LEM.AF_NCD_FC_III_1}
\lb\gamma,g\lc{}L_{x,N}^{\phi},R_{x,N}^{\bpsi}\rc\rb{}=0.    
\end{align}
\end{lem}
\begin{proof}
Since $K$ is compact and $g(t,s)=g(s,t)$ for all $t,s\in K$, approximate $g$ uniformly on $K\times K$ by symmetric polynomials. Thus reduce to $g$ polynomial by $\specN x\subset K$. Apply $\gamma$-symmetry as per Equation \ref{EQ.DFN.AF_Cstar_Bimodule_1}.
\end{proof}

\begin{cor}\label{COR.AF_NCD_FC_III}
Let $(\phi,\bpsi,\gamma)$ be $(N,V)$-compressible, $f$ symmetric. For all $x\in N_{+}$, get

\begin{itemize}
\item[1)] $\mathcal{M}_{x,N}^{\theta}\circ\gamma=\gamma\circ\mathcal{M}_{x,N}^{\theta}$,

\item[2)] $\mathcal{D}_{x,N}^{\theta}\circ\gamma=\gamma\circ\mathcal{D}_{x,N}^{\theta}$ if $x>0$ in $N$,

\item[3)] $\mathcal{D}_{x,N,\varepsilon}^{\theta}\circ\gamma=\gamma\circ\mathcal{D}_{x,N,\varepsilon}^{\theta}$ for all $\varepsilon>0$.
\end{itemize}
\end{cor}
\begin{proof}
By symmetry, get $m_{f}(t,s)^{\theta}=m_{f}(s,t)^{\theta}$ for all $t,s\geq 0$. Apply Lemma \ref{LEM.AF_NCD_FC_II}.
\end{proof}


\pagebreak


\begin{lem}\label{LEM.NCD_Operator_Compressed_PMO_I}
Let $(\phi,\bpsi,\gamma)$ be $(N,V)$-compressible.

\begin{itemize}
\item[1)] For all $x,y\in L^{0}(N,\tau)_{+}$, we have $\mathcal{M}_{x,y}^{\theta}\in\UBII_{V}\lc{}L^{2}(B,\omega)\rc$ and

\begin{align}\label{EQ.LEM.NCD_Operator_Compressed_PMO_I_1}
\restr{0.925}{\mathcal{M}_{x,y}^{\theta}}{V}=\mathcal{M}_{x,y,N}^{\theta}.
\end{align}

\begin{reapply}
\end{reapply}

\item[2)] For all $x,y>0$ in $L^{0}(N,\tau)$ and $\alpha,\beta>0$, we have $\mathcal{D}_{x,y,N}^{\theta}=\lc\restr{0.925}{\mathcal{M}_{x,y}}{V}\rc^{-\theta}$ and

\begin{align}\label{EQ.LEM.NCD_Operator_Compressed_PMO_I_2}
\restr{0.925}{\mathcal{D}_{x+\alpha 1_{N}^{\perp},y+\beta 1_{N}^{\perp}}^{\theta}}{V}=\restr{0.925}{\mathcal{M}_{x+\alpha 1_{N}^{\perp},y+\beta 1_{N}^{\perp}}^{-\theta}}{V}=\mathcal{M}_{x,y,N}^{-\theta}=\mathcal{D}_{x,y,N}^{\theta}.
\end{align}

\begin{reapply}
\end{reapply}

\item[3)] Let $N_{A}\subset \lc{}L^{\infty}(A,\tau),\tau\rc$ and $N_{B}\subset \lc{}L^{\infty}(B,\omega),\omega\rc$ be finite-dimensional s.t.~$1)$ and $2)$ in Lemma \ref{LEM.AF_Cstar_Bimodule_Compression_II} hold. For all $x,y>0$ in $N_{A}$ and $\alpha,\beta>0$, we have

\begin{align}\label{EQ.LEM.NCD_Operator_Compressed_PMO_I_3}
\mathcal{D}_{x+\alpha 1_{N_{A}}^{\perp},y+\beta 1_{N_{A}}^{\perp}}^{\theta}=\mathcal{D}_{x,y,N_{A}}^{\theta}\oplus m_{f}\lc\alpha,\beta\rc^{-\theta}I_{\vstretch{0.7125}{\big\langle}\hspace{-0.0275cm} 1_{N_{B}}^{\perp}\hspace{-0.03cm} \vstretch{0.7125}{\big\rangle}_{\mathbb{C}}}
\end{align}

\begin{reapply}
\end{reapply}

\noindent w.r.t.~$\BII(N_{B})\oplus\BII(\vstretch{0.8575}{\big\langle}\hspace{-0.0275cm} 1_{N_{B}}^{\perp}\hspace{-0.03cm} \vstretch{0.8575}{\big\rangle}_{\mathbb{C}})$.

\item[4)] Assume $(N,V)=\lc{}L^{\infty}(A[p],\tau),L^{2}(B[p],\omega)\rc$ for a projection $p\in L^{\infty}(A,\tau)$. For all $x,y\in L^{0}(A[p],\tau)_{+}$ and $\varepsilon>0$, we have

\begin{align}\label{EQ.LEM.NCD_Operator_Compressed_PMO_I_4}
\mathcal{D}_{x,y,\varepsilon}^{\theta}=\mathcal{D}_{x,y,p,\varepsilon}^{\theta}\oplus\lc\mathcal{D}_{x,0,\varepsilon}^{\theta}L_{p}^{\phi}\mathrlap{\phantom{R}^{\bpsi}}R_{p^{\perp}}+\mathcal{D}_{0,y,\varepsilon}^{\theta}\mathrlap{\phantom{L}^{\phi}}L_{p^{\perp}}R_{p}^{\bpsi}+\varepsilon^{-\theta}\pi_{p^{\perp}}\rc{}
\end{align}

\begin{reapply}
\end{reapply}

\noindent w.r.t.~$\BII\lc{}L^{2}(B[p],\omega)\rc\oplus\BII\lc{}L^{2}(B[p],\omega)^{\perp}\rc$.
\end{itemize}
\end{lem}
\begin{proof}
We have $1)$ by applying $2)$ in Lemma \ref{LEM.AF_NCD_FC_I} to $g=m_{f}^{\theta}$. We use $1)$ to obtain $2)$ by likewise application of $3)$ in Lemma \ref{LEM.AF_NCD_FC_I}. This uses strict positivity since application demands, for $g$ here, an extension from compact joint spectra to $[0,\infty)\times [0,\infty)$.\par
We show $3)$. Assume its setting. Let $x,y>0$ in $N_{A}$ and $\alpha,\beta>0$. Using $2)$, we have

\begin{align}\label{EQ.LEM.NCD_Operator_Compressed_PMO_I_5}
\mathcal{D}_{x+\alpha 1_{N_{A}}^{\perp},y+\beta 1_{N_{A}}^{\perp}}^{\theta}\pi_{N_{B}}^{B}=\mathcal{D}_{x,y,N_{A}}^{\theta}.
\end{align}

\noindent Note $N_{B}1_{N_{B}}^{\perp}=1_{N_{B}}^{\perp}N_{B}=0$ and $\phi\lc{}1_{N_{A}}^{\perp}\rc{}=\bpsi\lc{}1_{N_{A}}^{\perp}\rc{}=1_{N_{B}}^{\perp}$. Approximating $m_{f,\varepsilon}^{-\theta}$ uniformly using polynomials as in the proof of Lemma \ref{LEM.AF_NCD_FC_III}, we calculate

\begin{align}\label{EQ.LEM.NCD_Operator_Compressed_PMO_I_6}
\mathcal{D}_{x+\alpha 1_{N_{A}}^{\perp},y+\beta 1_{N_{A}}^{\perp}}^{\theta}\lc{}1_{N_{B}}^{\perp}\rc{}=m_{f}^{-\theta}\lc\alpha 1_{N_{B}}^{\perp},\beta 1_{N_{B}}^{\perp}\rc{}=m_{f}^{-\theta}\lc\alpha,\beta\rc{}1_{N_{B}}^{\perp}.
\end{align}

\noindent Using $2.3)$ in Lemma \ref{LEM.AF_NCD_FC_I}, Equation \ref{EQ.LEM.NCD_Operator_Compressed_PMO_I_5} and Equation \ref{EQ.LEM.NCD_Operator_Compressed_PMO_I_6} imply Equation \ref{EQ.LEM.NCD_Operator_Compressed_PMO_I_3} at once. Get $3)$.\par


\pagebreak


We show $4)$. Assume its setting. Let $x,y\in L^{0}(A[p],\tau)_{+}$ and $\varepsilon>0$. Note $m_{f,\varepsilon}\lc\varepsilon,\varepsilon\rc{}=f(1)\varepsilon=\varepsilon$ since $f(1)=1$. In addition, $1)$ in Proposition \ref{PRP.NCD_Operator_Compressed_PMO_II} implies

\begin{align}\label{EQ.LEM.NCD_Operator_Compressed_PMO_I_7}
\mathcal{D}_{x,y,\varepsilon}=m_{f,\varepsilon}^{-\theta}\lc{}L_{x}^{\phi},R_{y}^{\bpsi}\rc{},\ \mathcal{D}_{x,y,p,\varepsilon}=m_{f,\varepsilon}^{-\theta}\lc{}L_{x,p}^{\phi},R_{y,p}^{\bpsi}\rc{}.
\end{align}

\noindent Equation \ref{EQ.LEM.NCD_Operator_Compressed_PMO_I_7} shows $4)$ by Lemma \ref{LEM.AF_NCD_FC_II} and $2.3)$ in Lemma \ref{LEM.AF_NCD_FC_I} applied to $m_{f,\varepsilon}^{-\theta}$.
\end{proof}

Assume $A$ and $B$ are finite-dimensional. Let $B=r_{B}^{-1}\lc\oplus_{l=1}^{n}M_{n_{l}}(\mathbb{C})\rc$ equipped with its canonical AF-$B$-bimodule structure. The latter uses Notation \ref{NTN.AF_Cstar_Fin_Isometry}. Corollary \ref{COR.FC_Preservation} implies normal unital $^{*}$-homomorphisms preserve functional calculus. For all normal $z\in B$, get $\specB z=\bigcup_{l=1}^{n}\spec r_{B}(z)_{l}$. Thus $z$ is positive, resp.~strictly positive if and only if all $\lset{}r_{B}(z)_{l}\rset_{l=1}^{n}$ are. For all $x,y\in A_{h}$ and $g\in L^{\infty}\lc\specA x\times y,dE_{x,y,A}\rc$, we obtain

\begin{align}\label{EQ.SSEC.NCDS_NCD_QE_1}
g\lc{}L_{x}^{\phi},R_{y}^{\bpsi}\rc{}(u)=r_{B}^{-1}\lc\hspace{-0.025cm}\oplus_{l=1}^{n}g\lc{}L_{r_{B}\lc\phi(x)\rc_{l}},R_{r_{B}\lc\bpsi(y)\rc_{l}}\rc\big(r_{B}(u)_{l}\big)\rc{}
\end{align}

\noindent for all $u\in B$.

\begin{prp}\label{PRP.NCD_Operator_Compressed_PMO_Fin}
Assume $A$ and $B$ are finite-dimensional. Let $B=r_{B}^{-1}\lc\oplus_{l=1}^{n}M_{n_{l}}(\mathbb{C})\rc$ and equip $B$ with its canonical AF-$B$-bimodule structure. For all $x,y\in A_{+}$, we have

\begin{itemize}
\item[1)] $\mathcal{M}_{x,y}^{\theta}=r_{B}^{-1}\circ\lc\oplus_{l=1}^{n}\mathcal{M}_{r_{B}\lc\phi(x)\rc_{l},r_{B}\lc\bpsi(y)\rc_{l}}^{\theta}\rc\circ r_{B}$,

\item[2)] $\mathcal{D}_{x,y}^{\theta}=r_{B}^{-1}\circ\lc\oplus_{l=1}^{n}\mathcal{D}_{r_{B}\lc\phi(x)\rc_{l},r_{B}\lc\bpsi(y)\rc_{l}}^{\theta}\rc\circ r_{B}$ if $x,y>0$ in $A$.
\end{itemize}
\end{prp}
\begin{proof}
Equation \ref{EQ.SSEC.NCDS_NCD_QE_1} for $g=m_{f}^{\theta}$, resp.~$g=m_{f}^{-\theta}$.
\end{proof}


\subsubsection*{Quasi-entropies in the finite-dimensional setting}

Following the notion of monotone metric \cite{ART.Pet.1996.Monotone_Metrics}, quasi-entropies for full matrix algebras are given in \cite{ART.Hia_Pet.2012.Quasi_Entropy_I}\cite{ART.Hia_Pet.2013.Quasi_Entropy_II}. These are used to define energy functionals in \cite{ART.Car_Maa.2014.Quantum_OT_I}\cite{ART.Car_Maa.2017.Quantum_OT_II}\cite{ART.Car_Maa.2020.Quantum_OT_III}. Quasi-entropies, elsewhere known as quasi-entropy type functions instead, generalise quantum $f$-divergences \cite{ART.Hia.2018.QFD_I}\linebreak\cite{ART.Hia.2019.QFD_II}. We clarify and use terminology as per Remark \ref{REM.QE_Fin}.\par
Let $(A,\tau)$ and $(B,\omega)$ be finite-dimensional tracial AF-$C^{*}$-algebras. Let $(\phi,\bpsi,\gamma)$ be an AF-$A$-bimodule structure on $B$. Let $N\subset A$ and $V\subset B$ be a Hilbert subspace. Let $f$ be representing function of an operator mean and $\theta\in [0,1]$.

\begin{dfn}\label{DFN.QE_Fin_Strict}
We define functional $\mathcal{I}_{A,B}^{f,\theta}:A_{>0}\times A_{>0}\times B\longrightarrow [0,\infty)$ by setting

\begin{align}\label{EQ.DFN.QE_Fin_Strict_1}
\mathcal{I}_{A,B}^{f,\theta}(x,y,u):=\lgl\mathcal{D}_{x,y}^{\theta}(u),u\rgl_{\omega}
\end{align}

\noindent for all $x,y>0$ in $A$ and $u\in B$.
\end{dfn}

\begin{rem}\label{REM.QE_Fin}
In case of full matrix algebras, the terminology in both \cite{ART.Hia_Pet.2012.Quasi_Entropy_I} and \cite{ART.Hia_Pet.2013.Quasi_Entropy_II} is quasi-entropy type functions, rather than quasi-entropies. The latter are a special case for $\theta=-1$ fixed. We nevertheless use quasi-entropies consistent with \cite{ART.Car_Maa.2020.Quantum_OT_III}.
\end{rem}


\pagebreak


\begin{lem}\label{LEM.QE_Fin_I}
$\mathcal{I}^{f,\theta}$ is jointly convex. For all $u\in B$, the map $(x,y)\mapsto\mathcal{I}^{f,\theta}(x,y,u)$ decreases in partial order on $A_{>0}\times A_{>0}$ induced by pairs of positive elements.
\end{lem}
\begin{proof}
Proposition \ref{PRP.AF_Cstar_Trace_II} and $2)$ in Proposition \ref{PRP.NCD_Operator_Compressed_PMO_Fin} imply we have 

\begin{align}\label{EQ.LEM.QE_Fin_I_1}
\mathcal{I}_{A,B}^{f,\theta}(x,y,u)=\sum_{l=1}^{n}C_{l}\textrm{tr}_{n_{l}}\lc{}r_{B}(u)_{l}^{*}m_{f}^{-\theta}\lc{}L_{r_{B}\lc\phi(x)\rc_{l}},R_{r_{B}\lc\bpsi(y)\rc_{l}}\rc\big(r_{B}(u)_{l}\big)\rc{}
\end{align}

\noindent for all $x,y>0$ in $A$ and $u\in B$. If each summand on the right-hand side of Equation \ref{EQ.LEM.QE_Fin_I_1} satisfies the claimed properties, then our claims follow. We therefore reduce to the case of full matrix algebras since $\phi,\bpsi$ and $r_{B}$ are $^{*}$-homomorphisms.\par
Assume $A=B=M_{n}(\mathbb{C})$ for $n\in\mathbb{N}$ and $\phi=\bpsi=\id_{M_{n}(\mathbb{C})}$ without loss of generality. Note $\gamma$ is of no consequence here. Following \cite{ART.Hia_Pet.2013.Quasi_Entropy_II}, get the quasi-entropy type function

\begin{align}
\lc{}X,Y,U\rc\mapsto\lgl\mathcal{D}_{X,Y}(U),U\rgl_{\tr}=\textrm{tr}\lc{}U^{*}m_{f}^{-\theta}\lc{}L_{X},R_{Y}\rc{}(U)\rc{}
\end{align}

\noindent defined on $M_{n}(\mathbb{C})_{>0}\times M_{n}(\mathbb{C})_{>0}\times M_{n}(\mathbb{C})$. Theorem 2.1 in \cite{ART.Hia_Pet.2013.Quasi_Entropy_II} gives joint convexity of such functionals since $f$ is operator monotone and $\theta\in [0,1]$. We have operator mean $\lc{}X,Y\rc\mapsto \mathcal{M}_{X,Y}=m_{f}\lc{}L_{X},R_{Y}\rc$. Operator means are monotonically increasing on positive bounded operators \cite{ART.And_Kub.1979.Operator_Means}. Since inversion additionally reverts partial order on strictly positive bounded operators \lc{}cf.~Proposition \ref{PRP.Unbd_PO_Inversion}\rc{}, the map $\lc{}X,Y\rc\mapsto\mathcal{D}_{X,Y}$ decreases in partial order. Exponentiation with $\theta\in [0,1]$ preserves order, hence we obtain the map $\lc{}X,Y\rc\mapsto\textrm{tr}\lc{}U^{*}m_{f}^{-\theta}\lc{}L_{X},R_{Y}\rc{}(U)\rc$ decreases in partial order for all $U\in M_{n}(\mathbb{C})$.
\end{proof}

Identifying via musical isomorphisms, $A\cong A^{*}$ and $B\cong B^{*}$ as partially ordered vector spaces. Using $2)$ in Proposition \ref{PRP.NCD_Operator_Compressed_PMO_II}, we extend Equation \ref{EQ.DFN.QE_Fin_Strict_1} and therefore $\mathcal{I}_{A,B}^{f,\theta}$ to $A_{+}\cong A_{+}^{*}$ in the first two variables.

\begin{dfn}\label{DFN.QE_Fin}
We define quasi-entropy $\mathcal{I}_{A,B}^{f,\theta}:A_{+}^{*}\times A_{+}^{*}\times B^{*}\longrightarrow [0,\infty]$ by setting

\begin{align}\label{EQ.DFN.QE_Fin_1}
\mathcal{I}_{A,B}^{f,\theta}(\mu,\eta,w):=\sup_{\varepsilon>0}\hspace{0.025cm} \lgl\mathcal{D}_{\sharp\mu,\sharp\eta,\varepsilon}^{\theta}\lc\sharp w\rc{},\sharp w\rgl_{\omega}
\end{align}

\noindent for all $\mu,\eta\in A_{+}^{*}$ and $w\in B^{*}$.
\end{dfn}

\begin{ntn}
Let $\mathcal{I}_{B,B}^{f,\theta}$ denote the quasi-entropy for $B$ equipped with its canonical AF-$B$-bimodule structure.
\end{ntn}

\begin{prp}\label{PRP.QE_Fin_I}
$\mathcal{I}_{A,B}^{f,\theta}$ is jointly convex and l.s.c.~in $w^{*}$-topology.
\end{prp}
\begin{proof}
Lemma \ref{LEM.QE_Fin_I} shows joint convexity. For all $\varepsilon>0$, note $(x,y,u)\mapsto\lgl\mathcal{D}_{x,y,\varepsilon}^{\theta}(u),u\rgl_{\omega}$ defined on $A_{+}\times A_{+}\times B$ is norm continuous. Equation \ref{EQ.DFN.QE_Fin_1} shows l.s.c.~in $w^{*}$-topology by finite-dimensionality.
\end{proof}


\pagebreak


\begin{prp}\label{PRP.QE_Fin_II}
For all $x,y\in A_{+}$ and $u\in B$, we have

\begin{itemize}
\item[1)] $\mathcal{D}_{x,y,\varepsilon}^{\theta}=\mathcal{D}_{\phi(x),\bpsi(y),\varepsilon}^{\theta}$,

\item[2)] $\mathcal{I}_{A,B}^{f,\theta}\lc{}x^{\flat},y^{\flat},u^{\flat}\rc{}=\mathcal{I}_{B,B}\lc\phi(x)^{\flat},\bpsi(y)^{\flat},u^{\flat}\rc$.
\end{itemize}
\end{prp}
\begin{proof}
Apply $1)$ in Proposition \ref{PRP.NCD_Operator_Compressed_PMO_II} to get $1)$. The latter yields $2)$ by construction.
\end{proof}

\begin{lem}\label{LEM.QE_Fin_II}
Let $(\phi,\bpsi,\gamma)$ be $(N,V)$-compressible.

\begin{itemize}
\item[1)] For all $x,y\in N_{+}$ and $u\in V$, we have

\begin{align}\label{EQ.LEM.QE_Fin_II_1}
\mathcal{I}_{A,B}^{f,\theta}\lc{}x^{\flat},y^{\flat},u^{\flat}\rc{}=\sup_{\varepsilon>0}\hspace{0.025cm} \lgl\mathcal{D}_{x,y,N,\varepsilon}^{\theta}(u),u\rgl_{\omega}.
\end{align}

\begin{reapply}
\end{reapply}

\item[2)] Let $N_{A}\subset A$ and $N_{B}\subset B$ be $C^{*}$-subalgebras s.t.~$1)$ and $2)$ in Lemma \ref{LEM.AF_Cstar_Bimodule_Compression_II} hold. Let $\phi^{*}(N_{B}),\bpsi^{*}(N_{B})\subset N_{A}$. For all $\mu,\eta\in A_{+}^{*}$ and $w\in B^{*}$, we have

\begin{align}\label{EQ.LEM.QE_Fin_II_2}
\mathcal{I}_{N_{A},N_{B}}^{f,\theta}\lc\mu\vert_{N_{A}},\eta\vert_{N_{A}},w\vert_{N_{B}}\rc\leq\mathcal{I}_{A,B}^{f,\theta}(\mu,\eta,w).
\end{align}

\begin{reapply}
\end{reapply}

\end{itemize}
\end{lem}
\begin{proof}
In this proof, $\gamma$ is of no consequence. We have $1)$ at once by $2)$ in Lemma \ref{LEM.NCD_Operator_Compressed_PMO_I}. We show $2)$. Assume its setting. Note Remark \ref{REM.AF_Cstar_Trace_Fin}. We therefore consider $\lc{}N_{A},\tau\rc$ and $\lc{}N_{B},\omega\rc$ to be finite tracial AF-$C^{*}$-algebras.\par
We know $\phi,\bpsi:N_{A}\longrightarrow N_{B}$ are local $^{*}$-homomorphisms. We have AF-$N_{A}$-bimodule $(\phi,\bpsi,\gamma)$ on $N_{B}$. Set $\pi_{A}:=\pi_{N_{A}}^{A},\pi_{A,\textrm{u}}:=\pi_{N_{A}[1_{A}]}^{A}$ and $\pi_{B}:=\pi_{N_{B}}^{B},\pi_{B,\textrm{u}}:=\pi_{N_{B}[1_{B}]}^{B}$ here. Using $3.1)$ in Proposition \ref{PRP.AF_Cstar_Local_Hom_I}, we have

\begin{align}\label{EQ.LEM.QE_Fin_II_3}
\phi\vert_{N_{A}}\circ\pi_{A}=\pi_{B}\circ\phi\vert_{N_{A}},\ \bpsi\vert_{N_{A}}\circ\pi_{A}=\pi_{B}\circ\bpsi\vert_{N_{A}}.
\end{align}

\noindent Arguing as for $2.1)$ in Proposition \ref{PRP.AF_Cstar_Trace_Dualisation_I}, note identifying $A^{*}\cong A$ and $B^{*}\cong B$ via musical isomorphisms yields

\begin{align}\label{EQ.LEM.QE_Fin_II_4}
\textrm{res}_{N_{A}}=\pi_{A},\ \textrm{res}_{N_{B}}=\pi_{B} 
\end{align}

\noindent for restriction maps $\res_{N_{A}}:A^{*}\longrightarrow N_{A}$ and $\res_{N_{B}}:B^{*}\longrightarrow N_{B}$ obtained by dualising the given $C^{*}$-subalgebra inclusion maps. Finite-dimensionality shows we are in the setting of Proposition \ref{PRP.AF_Cstar_Trace_NCE_II}. Using Proposition \ref{PRP.AF_Cstar_Trace_NCE_II}, we see both noncommutative conditional expectations $\pi_{A}:A\longrightarrow N_{A}$ and $\pi_{B}:B\longrightarrow N_{B}$ decompose as

\begin{align}\label{EQ.LEM.QE_Fin_II_5}
\pi_{A}=\pi_{A,\textrm{u}}-\kappa_{N_{A}}^{A}1_{N_{A}}^{\perp},\ \pi_{B}=\pi_{B,\textrm{u}}-\kappa_{N_{B}}^{B}1_{N_{B}}^{\perp}.
\end{align}

Equation \ref{EQ.LEM.QE_Fin_II_3} and Equation \ref{EQ.LEM.QE_Fin_II_4} hold if we use $N_{A}[1_{A}]$ and $N_{B}[1_{B}]$ instead, i.e.~$\pi_{A,\textrm{u}}$ and $\pi_{B,\textrm{u}}$. Let $\mu,\eta\in A_{+}^{*}$ and $w\in B^{*}$. Set $x:=\sharp\mu$, $y:=\sharp\eta$ and $z:=\sharp w$. Equation \ref{EQ.LEM.QE_Fin_II_3} and Equation \ref{EQ.LEM.QE_Fin_II_4} show

\begin{align}\label{EQ.LEM.QE_Fin_II_7}
\pi_{B}\lc\phi(x)\rc{}=\phi\lc\pi_{A}(x)\rc{}=\phi\lc\sharp\mu\vert_{N_{A}}\rc{},\ \pi_{B}\lc\bpsi(y)\rc{}=\bpsi\lc\pi_{A}(y)\rc{}=\bpsi\lc\sharp\eta\vert_{N_{A}}\rc{}
\end{align}

\noindent and

\begin{align}\label{EQ.LEM.QE_Fin_II_8}
\pi_{B}(z)=\sharp w\vert_{N_{B}}.
\end{align}

\noindent We may use $N_{A}[1_{A}]$ and $N_{B}[1_{B}]$ instead. Using $1)$ in our setting, we see Equation \ref{EQ.LEM.QE_Fin_II_7} and Equation \ref{EQ.LEM.QE_Fin_II_8} show

\begin{align}\label{EQ.LEM.QE_Fin_II_9}
\mathcal{I}_{A,B}^{f,\theta}\lc\pi_{A}(x)^{\flat},\pi_{A}(y)^{\flat},\pi_{B}(z)^{\flat}\rc{}=\mathcal{I}_{N_{A},N_{B}}^{f,\theta}\lc\mu\vert_{N_{A}},\eta\vert_{N_{A}},w\vert_{N_{B}}\rc{}.
\end{align}

\noindent Equation \ref{EQ.LEM.QE_Fin_II_9} and $2)$ in Proposition \ref{PRP.QE_Fin_II} imply Equation \ref{EQ.LEM.QE_Fin_II_2} if for all $\varepsilon>0$, we have

\begin{align}\label{EQ.LEM.QE_Fin_II_10}
\lgl\mathcal{D}_{\pi_{A}(x),\pi_{A}(y),\varepsilon}^{\theta}\lc\pi_{B}(z)\rc{},\pi_{B}(z)\rgl_{\omega}\leq\lgl\mathcal{D}_{\pi_{A,\textrm{u}}(x),\pi_{A,\textrm{u}}(y),\varepsilon}^{\theta}\lc\pi_{B,\textrm{u}}(z)\rc{},\pi_{B,\textrm{u}}(z)\rgl_{\omega}
\end{align}

\noindent and

\begin{align}\label{EQ.LEM.QE_Fin_II_11}
\lgl\mathcal{D}_{\pi_{A,\textrm{u}}(x),\pi_{A,\textrm{u}}(y),\varepsilon}^{\theta}\lc\pi_{B,\textrm{u}}(z)\rc{},\pi_{B,\textrm{u}}(z)\rgl_{\omega}\leq\lgl\mathcal{D}_{x,y,\varepsilon}^{\theta}(z),z\rgl_{\omega}.
\end{align}

We show Equation \ref{EQ.LEM.QE_Fin_II_10}. Let $\varepsilon>0$. Using $3)$ in Lemma \ref{LEM.NCD_Operator_Compressed_PMO_I}, Equation \ref{EQ.LEM.QE_Fin_II_5}, as well as unitality of noncommutative conditional expectations for unital $C^{*}$-subalgebras, we see writing $\varepsilon 1_{A}=\varepsilon 1_{N}+\varepsilon 1_{N}^{\perp}$ yields

\begin{align}\label{EQ.LEM.QE_Fin_II_12}
\mathcal{D}_{\pi_{A,\textrm{u}}(x),\pi_{A,\textrm{u}}(y),\varepsilon}^{\theta}=\mathcal{D}_{\pi_{A}(x),\pi_{A}(y),N_{A},\varepsilon}\oplus m_{f}\lc\varepsilon+\kappa_{N_{A}}^{A}(x),\varepsilon+\kappa_{N_{A}}^{A}(y)\rc^{-\theta}I_{\vstretch{0.7125}{\big\langle}\hspace{-0.0275cm} 1_{N_{B}}^{\perp}\hspace{-0.03cm} \vstretch{0.7125}{\big\rangle}_{\mathbb{C}}}
\end{align}

\noindent w.r.t.~$\BII(N_{B})\oplus\BII(\vstretch{0.8575}{\big\langle}\hspace{-0.0275cm} 1_{N_{B}}^{\perp}\hspace{-0.03cm} \vstretch{0.8575}{\big\rangle}_{\mathbb{C}})$. Note $\mathcal{D}_{\pi_{A}(x),\pi_{A}(y),N_{A},\varepsilon}(N_{B})\subset N_{B}\subset\vstretch{0.8575}{\big\langle}\hspace{-0.0275cm} 1_{N_{B}}^{\perp}\hspace{-0.03cm} \vstretch{0.8575}{\big\rangle}_{\mathbb{C}}^{\perp}$. We obtain

\begin{align}\label{EQ.LEM.QE_Fin_II_13}
m_{f}\lc\varepsilon+\kappa_{N_{A}}^{A}(x),\varepsilon+\kappa_{N_{A}}^{A}(y)\rc^{-\theta}\kappa_{N_{B}}^{B}(z)\geq 0.
\end{align}

\noindent Using Equation \ref{EQ.LEM.QE_Fin_II_12} and Equation \ref{EQ.LEM.QE_Fin_II_13}, we estimate

\begin{align*}
& \ \lgl\mathcal{D}_{\pi_{A,\textrm{u}}(x),\pi_{A,\textrm{u}}(y),\varepsilon}^{\theta}\lc\pi_{B,\textrm{u}}(z)\rc{},\pi_{B,\textrm{u}}(z)\rgl_{\omega} \phantom{\bigg)} \\
=& \ \lgl\mathcal{D}_{\pi_{A}(x),\pi_{A}(y),\varepsilon}^{\theta}\lc\pi_{B}(z)\rc{},\pi_{B}(z)\rgl_{\omega}+m_{f}\lc\varepsilon+\kappa_{N_{A}}^{A}(x),\varepsilon+\kappa_{N_{A}}^{A}(y)\rc^{-\theta}\kappa_{N_{B}}^{B}(z)\cdot \|1_{N}^{\perp}\|_{\omega} \phantom{\bigg)} \\
\geq & \ \lgl\mathcal{D}_{\pi_{A}(x),\pi_{A}(y),\varepsilon}^{\theta}\lc\pi_{B}(z)\rc{},\pi_{B}(z)\rgl_{\omega}. \phantom{\bigg)}
\end{align*}

\noindent The above calculation shows Equation \ref{EQ.LEM.QE_Fin_II_10}.\par


\pagebreak


We show Equation \ref{EQ.LEM.QE_Fin_II_11}. Let $\varepsilon>0$. Using Equation \ref{EQ.LEM.QE_Fin_II_7} and Equation \ref{EQ.LEM.QE_Fin_II_8} for $N_{A}[1_{A}]$ and $N_{B}[1_{B}]$, $1)$ in Proposition \ref{PRP.QE_Fin_II} lets us calculate

\begin{align*}
& \ \lgl\mathcal{D}_{\pi_{A,\textrm{u}}(x),\pi_{A,\textrm{u}}(y),\varepsilon}^{\theta}\lc\pi_{B,\textrm{u}}(z)\rc{},\pi_{B,\textrm{u}}(z)\rgl_{\omega} \phantom{\bigg)} \\
=& \ \lgl\mathcal{D}_{\phi\lc\pi_{A,\textrm{u}}(x)\rc{},\bpsi\lc\pi_{A,\textrm{u}}(y)\rc{},\varepsilon}^{\theta}\lc\pi_{B,\textrm{u}}(z)\rc{},\pi_{B,\textrm{u}}(z)\rgl_{\omega} \phantom{\bigg)} \\
=& \ \lgl\mathcal{D}_{\pi_{B,\textrm{u}}\lc\phi(x)\rc{},\pi_{B,\textrm{u}}\lc\bpsi(y)\rc{},\varepsilon}^{\theta}\lc\pi_{B,\textrm{u}}(z)\rc{},\pi_{B,\textrm{u}}(z)\rgl_{\omega}. \phantom{\bigg)}
\end{align*}

\noindent Proposition \ref{PRP.AF_Cstar_Trace_NCE_I} shows

\begin{align}\label{EQ.LEM.QE_Fin_II_6_Rearranged}
\pi_{B,\textrm{u}}(v)=\int_{\UII(N_{B}')}uvu^{*}d\nu_{N_{B}}  
\end{align}

\noindent for all $v\in B$. Equation \ref{EQ.LEM.QE_Fin_II_6_Rearranged} expresses $\pi_{B}$ as average of unitary conjugations. Note the application of perturbed noncommutative division operators is jointly convex \lc{}cf.~proof of Lemma \ref{LEM.QE_Fin_I}\rc{}. We therefore apply the Jensen inequality \cite{ART.Per.1974.Vector_Valued_Jensen} to estimate

\begin{align}\label{EQ.LEM.QE_Fin_II_14}
\lgl\mathcal{D}_{\pi_{B,\textrm{u}}\lc\phi(x)\rc{},\pi_{B,\textrm{u}}\lc\bpsi(y)\rc{},\varepsilon}^{\theta}\lc\pi_{B,\textrm{u}}(z)\rc{},\pi_{B,\textrm{u}}(z)\rgl_{\omega}\leq\lgl\mathcal{D}_{\phi(x),\bpsi(y),\varepsilon}^{\theta}(z),z\rgl_{\omega}.
\end{align}

\noindent Altogether, Equation \ref{EQ.LEM.QE_Fin_II_14} and $1)$ in Proposition \ref{PRP.QE_Fin_II} imply Equation \ref{EQ.LEM.QE_Fin_II_11}.
\end{proof}

\begin{rem}\label{REM.QE_Monotonicity}
Equation \ref{EQ.LEM.QE_Fin_II_2} in Lemma \ref{LEM.QE_Fin_II} is the monotonicity of quasi-entropies under restriction maps, called monotonicity. We distinguish this from monotonicity of operators means implied by $2)$ in Proposition \ref{PRP.NCD_Operator_Compressed_PMO_II}. 
\end{rem}

\begin{lem}\label{LEM.QE_Fin_III}
Assume $f$ is symmetric. For all $x,y>0$ in $A$ and $u\in B$, we have

\begin{itemize}
\item[1)] $\|u\|_{\omega}^{2}\leq\mathcal{I}_{A,B}^{f,\theta}\lc{}x^{\flat},y^{\flat},u^{\flat}\rc\cdot 2^{-\theta}\lc\|x\|_{\infty}^{\theta}+\hspace{0.025cm}\| y\|_{\infty}^{\theta}\rc$,

\item[2)] $\|u\|_{1}^{2}\leq\mathcal{I}_{A,B}^{f,\theta}\lc{}x^{\flat},y^{\flat},u^{\flat}\rc\cdot 2^{-\theta}\lc\|\phi\|_{1}^{\theta}\|x\|_{1}^{\theta}+\|\bpsi\|_{1}^{\theta}\| y\|_{1}^{\theta}\rc\cdot \omega\lc{}1_{B}\rc^{1-\theta}$.
\end{itemize}
\end{lem}
\begin{proof}
The arithmetic operator mean is the maximal symmetric one \lc{}cf.~Theorem 4.5 in \cite{ART.And_Kub.1979.Operator_Means}\rc{}. Note $r\mapsto r^{\theta}$ on $[0,\infty)$ preserves order. For all $x,y>0$ in $A$, get 

\begin{align}\label{EQ.LEM.QE_Fin_III_1}
\mathcal{M}_{x,y}^{\theta}=m_{f}^{\theta}\lc{}L_{x}^{\phi},R_{y}^{\bpsi}\rc\leq 2^{-\theta}\big(L_{x}^{\phi}+R_{y}^{\bpsi}\big)^{\theta}.
\end{align}

\noindent For all $x,y>0$ in $A$ and $u\in B$, apply $L_{x}^{\phi}+R_{y}^{\bpsi}\leq\big(\|\phi\|_{\infty}\|x\|_{\infty}+\|\bpsi\|_{\infty}\| y\|_{\infty}\big)\cdot I$ to estimate

\begin{align}\label{EQ.LEM.QE_Fin_III_2}
\|u\|_{\omega}^{2}\leq \dblv{}\mathcal{M}_{x,y}^{\theta}\dblv{}\cdot \dblv{}\mathcal{D}_{x,y}^{\frac{\theta}{2}}(u)\dblv_{\omega}^{2}\leq\mathcal{I}_{A,B}^{f,\theta}\lc{}x^{\flat},y^{\flat},u^{\flat}\rc\cdot 2^{-\theta}\lc\|\phi\|_{\infty}\|x\|_{\infty}+\|\bpsi\|_{\infty}\| y\|_{\infty}\rc^{\theta}
\end{align}

\noindent using Equation \ref{EQ.LEM.QE_Fin_III_1}. Note $\|\phi\|_{\infty}=\|\bpsi\|_{\infty}=1$ by $2)$ in Lemma \ref{LEM.AF_Cstar_Local_Hom}. Further, $r\mapsto r^{\theta}$ is concave and therefore subadditive on $[0,\infty)$ since $\theta\in [0,1]$. Equation \ref{EQ.LEM.QE_Fin_III_2} shows $1)$.\par


\pagebreak


We prove $2)$. For all $x,y>0$ in $A$ and $u\in B$, we use the maximal symmetric operator mean property as above to estimate

\begin{align}\label{EQ.LEM.QE_Fin_III_3}
\babsv{1}{\omega\lc{}u^{*}z\rc{}}^{2}\leq \dblv{}\mathcal{D}_{x,y}^{\frac{\theta}{2}}(u)\dblv_{\omega}^{2}\cdot \dblv{}\mathcal{M}_{x,y}^{\frac{\theta}{2}}(z)\dblv_{\omega}^{2}\leq \dblv{}\mathcal{D}_{x,y}^{\frac{\theta}{2}}(u)\dblv_{\omega}^{2}\cdot 2^{-\theta}\lgl\big(L_{x}^{\phi}+R_{y}^{\bpsi}\big)^{\theta}(z),z\rgl_{\omega}.
\end{align}

\noindent Subadditivity of $r\mapsto r^{\theta}$ implies $\lc{}S+T\rc^{\theta}\leq S^{\theta}+T^{\theta}$ for commuting bounded operators $T,S\geq 0$ by functional calculus. Since $L^{\phi}$ and $R^{\bpsi}$ are $^{*}$-representations, we obtain

\begin{align}\label{EQ.LEM.QE_Fin_III_4}
\lgl\big(L_{x}^{\phi}+R_{y}^{\bpsi}\big)^{\theta}(z),z\rgl_{\omega}\leq\lgl\phi(x)^{\theta}(z),z\rgl_{\omega}+\lgl z\bpsi(y)^{\theta},z\rgl_{\omega}\leq\big(\|\phi(x)^{\theta}\|_{1}+\|\bpsi(y)^{\theta}\|_{1}\big)\cdot \| z\|_{B}^{2}.
\end{align}

\noindent For all $v\in B_{+}$ and $\theta\in [0,1]$, $\|v^{\theta}\|_{1}\leq\omega\lc{}1_{B}\rc^{1-\theta}\| v\|_{1}^{\theta}$ by Jensen's inequality. Equation \ref{EQ.LEM.QE_Fin_III_3} and Equation \ref{EQ.LEM.QE_Fin_III_4} together show $2)$ as norm is obtained by testing on $B$. For this, note $\dblv{}\phi(x)\dblv_{1}\leq \|\phi\|_{1}\|x\|_{1}$ and $\dblv{}\bpsi(y)\dblv_{1}\leq \|\bpsi\|_{1}\| y\|_{1}$.
\end{proof}


\subsubsection*{Extending to AF-$C^{*}$-bimodules}

Monotonicity extends quasi-entropies from the finite-dimensional setting to the AF-$C^{*}$-setting. Theorem \ref{THM.QE_AF} collects fundamental properties. We view each symmetric representing function $f$ as determining a class of energetic structures with $\theta\in [0,1]$ as interpolation parameter. Proposition \ref{PRP.QOT_Distance_Interpolation_Parameter} shows $\theta=0$ gives quantum $\lc{}-1,2\rc$-Sobolev distance independent of $f$. In the logarithmic mean setting, i.e.~$f$ represents the logarithmic operator mean and $\theta=1$, we obtain quantum $L^{2}$-Wasserstein distances in direct analogy to the classical case \cite{ART.Dol_Naz_Sav.2009.Generalised_OT}.\par
Let $(A,\tau)$ and $(B,\omega)$ be tracial AF-$C^{*}$-algebras. Let $(\phi,\bpsi,\gamma)$ be an AF-$A$-bimodule structure on $B$. Let $f$ be representing function of an operator mean and $\theta\in [0,1]$. For all $j\in\mathbb{N}$, we use induced AF-$A_{j}$-bimodule structure on $B_{j}$ as per $4)$ in Definition \ref{DFN.AF_Cstar_Bimodule}.

\begin{dfn}\label{DFN.QE_AF}\hspace{1cm}
\begin{itemize}
\item[1)] For all $j\in\mathbb{N}$, we call $\mathcal{I}_{A,B,j}^{f,\theta}:=\mathcal{I}_{A_{j},B_{j}}^{f,\theta}$ the $j$-th restricted quasi-entropy.

\item[2)] We define quasi-entropy $\mathcal{I}_{A,B}^{f,\theta}:A_{+}^{*}\times A_{+}^{*}\times B^{*}\longrightarrow [0,\infty]$ by setting

\begin{align}\label{EQ.DFN.QE_AF_1}
\mathcal{I}_{A,B}^{f,\theta}(\mu,\eta,w):=\sup_{j\in\mathbb{N}}\hspace{0.025cm} \mathcal{I}_{A,B,j}^{f,\theta}\lc\mu_{j},\eta_{j},w_{j}\rc{}
\end{align}

\begin{reapply}
\end{reapply}

\noindent for all $\mu,\eta\in A_{+}^{*}$ and $w\in B^{*}$.
\end{itemize}
\end{dfn}

\begin{ntn}\label{NTN.QE_AF}
Unless stated otherwise, we suppress $A$ and $B$ in Definition \ref{DFN.QE_AF}.
\end{ntn}

\begin{cor}\label{COR.QE_AF}
For all $j\leq k$ in $\mathbb{N}$, we have

\begin{itemize}
\item[1)] $\mathcal{I}_{j}^{f,\theta}(\mu,\eta,w)=\mathcal{I}_{k}^{f,\theta}(\mu,\eta,w)$ for all $\mu,\eta\in A_{j,+}^{*}$ and $w\in B_{j}^{*}$,

\item[2)] $\mathcal{I}_{j}^{f,\theta}\lc\mu_{j},\eta_{j},w_{j}\rc\leq\mathcal{I}_{k}^{f,\theta}(\mu,\eta,w)$ for all $\mu,\eta\in A_{k,+}^{*}$ and $w\in B_{k}^{*}$.
\end{itemize}
\end{cor}
\begin{proof}
For all $j\leq k$ in $\mathbb{N}$, apply Lemma \ref{LEM.QE_Fin_II} to $N_{A}=A_{j}$ and $N_{B}=B_{j}$ in the setting of the induced AF-$A_{k}$-bimodule $B_{k}$. This shows both claims at once.
\end{proof}


\pagebreak


\begin{dfn}\label{DFN.QE_AF_Trace_Dualisation}
For all $j\in\mathbb{N}$, we define

\begin{itemize}
\item[1)] $\incj:A_{j,+}^{*}\times A_{j,+}^{*}\times B_{j}^{*}\longrightarrow A_{+}^{*}\times A_{+}^{*}\times B^{*}$ by setting 

\begin{align}\label{EQ.DFN.QE_AF_Trace_Dualisation_1}
\incj(\mu,\eta,w):=(\mu,\eta,w)    
\end{align}

\begin{reapply}
\end{reapply}

\noindent for all $\mu,\eta\in A_{j,+}^{*}$ and $w\in B_{j}^{*}$,

\item[2)] $\resj:A_{+}^{*}\times A_{+}^{*}\times B^{*}\longrightarrow A_{j,+}^{*}\times A_{j,+}^{*}\times B_{j}^{*}$ by setting 

\begin{align}\label{EQ.DFN.QE_AF_Trace_Dualisation_2}
\resj(\mu,\eta,w):=\lc\mu_{j},\eta_{j},w_{j}\rc{}    
\end{align}

\begin{reapply}
\end{reapply}

\noindent for all $\mu,\eta\in A_{+}^{*}$ and $w\in B^{*}$.
\end{itemize}
\end{dfn}

\begin{thm}\label{THM.QE_AF}
Let $(A,\tau)$ and $(B,\omega)$ be tracial AF-$C^{*}$-algebras. Let $(\phi,\bpsi,\gamma)$ be an AF-$A$-bimodule structure on $B$. Let $f$ be representing function of an operator mean and $\theta\in [0,1]$.

\begin{itemize}
\item[1)] $\mathcal{I}^{f,\theta}$ is jointly convex and l.s.c.~in $w^{*}$-topology. \phantom{\bigg)}

\item[2)] $\mathcal{I}_{j}^{f,\theta}=\mathcal{I}_{k}^{f,\theta}\circ\inckj=\mathcal{I}^{f,\theta}\circ\incj$ for all $j\leq k$ in $\mathbb{N}$. \phantom{\bigg)}

\item[3)] $\mathcal{I}_{j}^{f,\theta}\circ\resj\leq \mathcal{I}_{k}^{f,\theta}\circ\resk$ for all $j\leq k$ in $\mathbb{N}$. \phantom{\bigg)}

\item[4)] Assume $f$ is symmetric. For all $\mu,\eta\in A_{+}^{*}\cap L^{\infty}(A,\tau)^{\flat}$ and $w\in B^{*}\cap L^{2}(B,\omega)^{\flat}$, we have

\begin{align}\label{EQ.THM.QE_AF_1}
\|\sharp w\|_{\omega}^{2}\leq\mathcal{I}^{f,\theta}(\mu,\eta,w)\cdot 2^{-\theta}\lc\|\sharp\mu\|_{\infty}^{\theta}+\|\sharp\eta\|_{\infty}^{\theta}\rc{}.
\end{align}

\begin{reapply}
\end{reapply}

\item[5)] Assume $f$ is symmetric. For all $\mu,\eta\in A_{+}^{*}$ and $w\in B^{*}$, we have

\begin{align}\label{EQ.THM.QE_AF_2}
\|w\|_{B^{*}}^{2}\leq\mathcal{I}^{f,\theta}(\mu,\eta,w)\cdot 2^{-\theta}\lc\|\phi\|_{1}^{\theta}\|\mu\|_{A^{*}}^{\theta}+\|\bpsi\|_{1}^{\theta}\|\eta\|_{A^{*}}^{\theta}\rc\cdot \|\omega\|^{1-\theta}
\end{align}

\begin{reapply}
\end{reapply}

\noindent with $\|\omega\|:=\omega\lc{}1_{B}\rc{}=\sup_{j\in\mathbb{N}}\omega(1_{B_{j}})$ the volume and convention $\|\omega\|^{0}:=1$.
\end{itemize}
\end{thm}
\begin{proof}
Since restriction is $w^{*}$-continuous, Proposition \ref{PRP.QE_Fin_I} implies $1)$. Get $2)$ and $3)$ by Corollary \ref{COR.QE_AF}. Following $2.1)$ in Proposition \ref{PRP.AF_Cstar_Trace_Dualisation_II}, all noncommutative $L^{1}$-, $L^{2}$-~and $L^{\infty}$-norms in use are the suprema over $j\in\mathbb{N}$ of their restrictions to $A_{j}$, resp.~$B_{j}$. Writing norms as such, get $4)$ and $5)$ by Lemma \ref{LEM.QE_Fin_III}.
\end{proof}

\begin{rem}
We know $\|\phi\|_{1},\|\bpsi\|_{1}<\infty$ by $2.1)$ in Lemma \ref{LEM.AF_Cstar_Local_Hom}. If $(A,\tau)=(B,\omega)$ with self-adjoint local $^{*}$-homomorphisms, then $\|\phi\|_{1}=\|\bpsi\|_{1}\leq 2$. If $\theta=1$, then the volume term in Equation \ref{EQ.THM.QE_AF_2} vanishes. This allows estimates for unbounded traces.
\end{rem}


\subsection{Noncommutative division operators from quasi-entropies}\label{SSEC.NCDS_NCD_Operators}

Quasi-entropies determine closed positive unbounded quadratic forms on symmetric $W^{*}$-bimodules given pairs of positive bounded functionals on tracial AF-$C^{*}$-algebras. Theorem \ref{THM.QE_QForm_Representation} shows such quadratic forms have unique representing operators. These are, by definition, noncommutative division operators of positive bounded functionals on tracial AF-$C^{*}$-algebras. Under the modified standard pairing, normal positive bounded functionals on tracial AF-$C^{*}$-algebras are positive measurable operators. Using results in Theorem \ref{THM.AF_Cstar_Bimodule_CLRA_SR}, Theorem \ref{THM.NCD_Operator_Compressed_PMO} states necessary and sufficient conditions to recover noncommutative division operators of positive measurable operators.\par
We construct noncommutative division operators as follows using the Kato-Robinson theorem \lc{}cf.~Theorem 10.4.2 in \cite{BK.deOli.2009.OpAlg_Quantum_Dynamics}\rc{}. We define perturbed left-~and right-division with positive bounded functionals on tracial AF-$C^{*}$-algebras. Inverses exist and are strongly commuting positive unbounded operators. Using their joint spectral calculus, we define perturbed noncommutative division operators in direct analogy to positive measurable operators. Theorem \ref{THM.QE_QForm_Representation} shows strong resolvent limits exist as perturbation tends to zero. These limits are noncommutative division operators as above. Standard reference for unbounded quadratic forms and the Kato-Robinson theorem is \cite{BK.deOli.2009.OpAlg_Quantum_Dynamics}.


\subsubsection*{Unbounded quadratic forms and the Kato-Robinson theorem}

Let $H$ be a Hilbert space. The Kato-Robinson theorem relates closed positive unbounded quadratic forms, their representing operators and strong resolvent limits as follows.\par
The full Kato-Robinson theorem, i.e.~its general formulation, uses strong resolvent convergence of positive unbounded operators on Hilbert subspaces. Definition \ref{DFN.SR_Subspace} generalises Definition \ref{DFN.SR} accordingly. Proposition \ref{PRP.SR_Subspace} shows uniform reducibility lets us restrict again to strong resolvent convergence on Hilbert spaces. For details on strong resolvent convergence on Hilbert spaces, we refer to Subsection \ref{SSEC.A_Maps_SR}. For details on reducing subspaces, we refer to Subsection \ref{SSEC.A_Maps_Compression}.

\begin{dfn}\label{DFN.SR_Subspace}
Let $V\subset H$ be a Hilbert subspace. We call $\lset{}T_{n}\rset_{n\in\mathbb{N}}\subset\UBII(H)_{+}$ strong resolvent convergent to $T\in\UBII\lc{}V\rc_{+}$ on $V$ in $H$ if for all $a>0$, we have

\begin{align}\label{EQ.DFN.SR_Subspace_1}
R_{-a}(T)(u)=\|.\|_{V}\textrm{-}\lim_{n\in\mathbb{N}}\hspace{0.025cm} R_{-a}(T_{n})\big(\pi_{V}(u)\big)
\end{align}

\noindent for all $u\in H$.
\end{dfn}

\begin{ntn}\label{NTN.SR_Subspace}
Let $T=\sr$-$\lim_{n\in\mathbb{N}}T_{n}$ on $V$ in $H$ denote strong resolvent convergence. We drop \textit{on V} if $V$ is clear from context, resp.~\textit{in H} if $H$ is.
\end{ntn}

\begin{rem}
The resolvents in Equation \ref{EQ.DFN.SR_Subspace_1} are given by bounded measurable functional calculus in a priori different $W^{*}$-algebras. If $V=H$, then Equation \ref{EQ.DFN.SR_Subspace_1} is strong convergence and we recover Definition \ref{DFN.SR} for positive unbounded operators by Lemma \ref{LEM.FC_SR} and $1)$ in Proposition \ref{PRP.SR}.
\end{rem}

\begin{prp}\label{PRP.SR_Subspace}
If $T=\sr$-$\lim_{n\in\mathbb{N}}T_{n}$ on $V$ in $H$ and $\lset{}T_{n}\rset_{n\in\mathbb{N}}\subset\UBII_{V}(H)$, then we have $T=\sr$-$\lim_{n\in\mathbb{N}}\restr{0.925}{T_{n}}{V}$ on $V$.
\end{prp}
\begin{proof}
Using $1)$ in Proposition \ref{PRP.Reducible} and $2)$ in Lemma \ref{LEM.Compression_Preservation_I}, Equation \ref{EQ.DFN.SR_Subspace_1} for fixed but arbitrary $a>0$ restricts to 

\begin{align}\label{EQ.PRP.SR_Subspace_1}
R_{-a}(T)=\s\textrm{-}\lim_{n\in\mathbb{N}}R_{-a}\lc\restr{0.925}{T_{n}}{V}\rc.
\end{align}

\noindent Using $1)$ in Proposition \ref{PRP.SR}, Equation \ref{EQ.PRP.SR_Subspace_1} shows $T=\sr$-$\lim_{n\in\mathbb{N}}\restr{0.925}{T_{n}}{V}$ on $V$.
\end{proof}

\begin{dfn}
Let $H$ be a separable Hilbert space. For all positive unbounded quadratic forms $Q:H\longrightarrow [0,\infty]$, set

\begin{itemize}
\item[1)] $\dom Q:=\big\{\hspace{0.025cm} u\in H\ \vset\ Q(u)<\infty\hspace{0.025cm} \big\}$,

\item[2)] $H\lc{}Q\rc{}:=\overline{\dom Q}^{\|.\|_{H}}$.
\end{itemize}
\end{dfn}

Let $H$ be a Hilbert space. Theorem 9.3.7 in \cite{BK.deOli.2009.OpAlg_Quantum_Dynamics} gives positivity-preserving bijection between closed positive unbounded quadratic forms and representing operators. If $Q$ is a closed positive unbounded quadratic form on $H$, then it has representing operator $T\in\UBII\lc{}H\lc{}Q\rc\rc_{+}$ s.t.~$\dom Q=\dom\sqrt{T}$ and

\begin{align}\label{EQ.SSEC.NCDS_NCD_Operators_1}
Q(u)=\lgl\sqrt{T}(u),\sqrt{T}(u)\rgl_{H}
\end{align}

\noindent for all $u\in\dom Q$. For all monotonically increasing sequences $\lset{}T_{n}\rset_{n\in\mathbb{N}}\subset\BII(H)_{+}$, we define closed positive unbounded quadratic form on $H$ by setting

\begin{align}\label{EQ.SSEC.NCDS_NCD_Operators_2}
Q(u):=\sup_{n\in\mathbb{N}}\hspace{0.025cm}\lgl T_{n}(u),u\rgl_{H}\in [0,\infty] 
\end{align}

\noindent for all $u\in H$. If $T$ is its representing operator, then $T=\sr$-$\lim_{n\in\mathbb{N}}T_{n}$ on $H\lc{}Q\rc$ by the Kato-Robinson theorem. Remark \ref{REM.QE_QForm_SR} below fixes conventions for using uncountable monotonically decreasing sequences instead.

\begin{rem}\label{REM.QE_QForm_SR}
Consider monotonically increasing $\lset{}T_{\varepsilon}\rset_{\varepsilon>0}\subset\BII(H)_{+}$ in dual order. All sequences $\lset{}T_{\varepsilon_{n}}\rset_{n\in\mathbb{N}}$ given fixed but arbitrary monotonically decreasing $\lset\varepsilon_{n}\rset_{n\in\mathbb{N}}\subset (0,\infty)$ generate identical quadratic form as per Equation \ref{EQ.SSEC.NCDS_NCD_Operators_2}. Uniqueness ensures each has identical strong resolvent limit, denoted by $T=\sr$-$\lim_{\varepsilon\downarrow 0}T_{\varepsilon}$ in this case.
\end{rem}


\subsubsection*{The unbounded operator representation of quasi-entropies}

Let $(A,\tau)$ and $(B,\omega)$ be tracial AF-$C^{*}$-algebras. Let $(\phi,\bpsi,\gamma)$ be an AF-$A$-bimodule structure on $B$. Let $f$ be representing function of an operator mean and $\theta\in [0,1]$. For all $\mu,\eta\in A_{+}^{*}$, we extend the map $u\mapsto\mathcal{I}^{f,\theta}(\mu,\eta,u)$ from $B_{0}$ to $L^{2}(B,\omega)$. Such extensions determine closed positive unbounded quadratic forms on $L^{2}(B,\omega)$. We equip $A_{+}^{*}\times A_{+}^{*}\times L^{2}(B,\omega)$ with the product topology of the given $w^{*}$-topologies. We use Notation \ref{NTN.AF_Cstar_Trace_Dualisation}.\par


\pagebreak


\begin{dfn}\label{DFN.Quadratic_Form}\hspace{1cm}

\begin{itemize}
\item[1)] For all $\mu,\eta\in A_{+}^{*}$, set

\begin{itemize}
\item[1.1)] $Q_{\mu,\eta}^{f,\theta}(u):=\sup_{j\in\mathbb{N}}\mathcal{I}_{j}^{f,\theta}\lc\mu_{j},\eta_{j},u_{j}\rc$ for all $u\in L^{2}(B,\omega)$, 

\item[1.2)] $\dom Q_{\mu,\eta}^{f,\theta}:=\big\{\hspace{0.025cm} u\in L^{2}(B,\omega)\ \vset\ Q_{\mu,\eta}^{f,\theta}(u)<\infty\hspace{0.025cm} \big\}$.
\end{itemize}

\begin{reapply}
\end{reapply}

\item[2)] We define $Q^{f,\theta}:A_{+}^{*}\times A_{+}^{*}\times L^{2}(B,\omega)\longrightarrow [0,\infty]$ by setting $Q^{f,\theta}(\mu,\eta,u):=Q_{\mu,\eta}^{f,\theta}(u)$ for all $\mu,\eta\in A_{+}^{*}$ and $u\in L^{2}(B,\omega)$.
\end{itemize}
\end{dfn}

\begin{rem}\label{REM.Quadratic_Form}
Note $L^{1,2}(B,\omega)\subset L^{1}(B,\omega)\subset B^{*}$. Further note $\mathcal{I}^{f,\theta}$ and $Q^{f,\theta}$ coincide on $A_{+}^{*}\times A_{+}^{*}\times L^{1,2}(B,\omega)$. If $\omega<\infty$, then $L^{2}(B,\omega)=L^{1,2}(B,\omega)$ and $Q^{f,\theta}$ is the restriction of $\mathcal{I}^{f,\theta}$ to $A_{+}^{*}\times A_{+}^{*}\times L^{2}(B,\omega)$. 
\end{rem}

\begin{prp}\label{PRP.Quadratic_Form}
We have

\begin{itemize}
\item[1)] $Q^{f,\theta}$ is jointly convex and l.s.c.~in $w^{*}$-topology,

\item[2)] $Q^{f,\theta}\circ\incj=\mathcal{I}_{j}^{f,\theta}$ for all $j\in\mathbb{N}$.
\end{itemize}
\end{prp}
\begin{proof}
Get $1)$ and $2)$ by arguing as for $1)$, resp.~$2)$ in Theorem \ref{THM.QE_AF}.
\end{proof}

We construct perturbed left-~and right-division by positive bounded functionals. For all $\mu\in A_{+}^{*}$, $\varepsilon>0$ and $j\in\mathbb{N}$, as well as fix but arbitrary $\eta\in A_{+}^{*}$, we have positive bounded quadratic form on $L^{2}(B,\omega)$ defined by

\begin{align}\label{EQ.SSEC.NCDS_NCD_Operators_3}
\mathbf{L}_{\mu_{j},\varepsilon}^{-\phi}(u):=Q_{\mu_{j}+\varepsilon 1_{A_{j}},\eta_{j}+\varepsilon 1_{A_{j}}}^{t,1}(u_{j})=\lgl\pi_{j}^{B}\lc\lc\phi\lc\mu_{j}\rc{}+\varepsilon I\rc^{-1}\pi_{j}^{B}(u)\rc{},u\rgl_{\omega}
\end{align}

\noindent for all $u\in L^{2}(B,\omega)$ using $(t,s)\mapsto t$ as our representing function. The right-hand side of Equation \ref{EQ.SSEC.NCDS_NCD_Operators_3} does not depend on $\eta\in A_{+}^{*}$. For all $j\leq k$ in $\mathbb{N}$, get $\pi_{jk}^{B}(1_{B_{k}})=1_{B_{j}}$. Thus $2)$ in Proposition \ref{PRP.Quadratic_Form} and $3)$ in Theorem \ref{THM.QE_AF} yield monotonically increasing sequence of uniformly positive and bounded quadratic forms on $L^{2}(B,\omega)$ s.t.~

\begin{align}\label{EQ.SSEC.NCDS_NCD_Operators_4}
0\leq\mathbf{L}_{\mu_{1},\varepsilon}^{-\phi}\leq\ldots\leq\mathbf{L}_{\mu_{j},\varepsilon}^{-\phi}\leq\mathbf{L}_{\mu_{j+1},\varepsilon}^{-\phi}\leq\ldots\leq\varepsilon^{-1}I.
\end{align}

\noindent Note Equation \ref{EQ.SSEC.NCDS_NCD_Operators_4} gives monotonically increasing sequence $\lset\pi_{j}^{B}(L_{\phi\lc\mu_{j}\rc{}}+\varepsilon I)^{-1}\pi_{j}^{B}\rset_{j\in\mathbb{N}}$ of uniformly positive and bounded operators as determined by Equation \ref{EQ.SSEC.NCDS_NCD_Operators_3}. Hence the Kato-Robinson theorem shows its strong limit is the unique positive bounded operator representing the positive bounded quadratic form defined by

\begin{align}\label{EQ.SSEC.NCDS_NCD_Operators_5}
\mathbf{L}_{\mu,\varepsilon}^{-\phi}(u):=\sup_{j\in\mathbb{N}}\hspace{0.025cm} \mathbf{L}_{\mu_{j},\varepsilon}^{-\phi}(u)=\sup_{j\in\mathbb{N}}\hspace{0.025cm} \lgl\pi_{j}^{B}\lc\lc\phi\lc\mu_{j}\rc{}+\varepsilon I\rc^{-1}\pi_{j}^{B}(u)\rc{},u\rgl_{\omega}
\end{align}

\noindent for all $\mu\in A_{+}^{*}$, $\varepsilon>0$ and $u\in L^{2}(B,\omega)$.\par


\pagebreak


We analogously construct perturbed right-division using $(t,s)\mapsto s$ as representing function. For all $\eta\in A_{+}^{*}$, $\varepsilon>0$ and $j\in\mathbb{N}$, we have positive bounded quadratic form on $L^{2}(B,\omega)$ defined by

\begin{align}\label{EQ.SSEC.NCDS_NCD_Operators_6}
\mathbf{R}_{\eta_{j},\varepsilon}^{-\bpsi}(u):=\lgl\pi_{j}^{B}\lc\pi_{j}^{B}(u)\lc\bpsi\lc\eta_{j}\rc{}+\varepsilon I\rc^{-1}\rc{},u\rgl_{\omega}
\end{align}

\noindent for all $u\in L^{2}(B,\omega)$. As above, we have monotonically increasing sequence of uniformly positive and bounded operators as determined by Equation \ref{EQ.SSEC.NCDS_NCD_Operators_6}. The Kato-Robinson theorem shows its strong limit is the unique positive bounded operator representing the positive bounded quadratic form defined by

\begin{align}\label{EQ.SSEC.NCDS_NCD_Operators_7}
\mathbf{R}_{\eta,\varepsilon}^{-\bpsi}(u):=\sup_{j\in\mathbb{N}}\hspace{0.025cm} \mathbf{R}_{\eta_{j},\varepsilon}^{-\bpsi}(u)=\sup_{j\in\mathbb{N}}\hspace{0.025cm} \lgl\pi_{j}^{B}\lc\pi_{j}^{B}(u)\lc\bpsi\lc\eta_{j}\rc{}+\varepsilon I\rc^{-1}\rc{},u\rgl_{\omega}
\end{align}

\noindent for all $\eta\in A_{+}^{*}$, $\varepsilon>0$ and $u\in L^{2}(B,\omega)$.

\begin{rem}
Note the Kato-Robinson theorem by itself only implies strong resolvent convergence. Using Proposition 10.1.13 in \cite{BK.deOli.2009.OpAlg_Quantum_Dynamics}, we know uniform boundedness together with strong resolvent convergence implies strong convergence. If uniform boundedness is given when applying the Kato-Robinson theorem, then we have strong convergence.
\end{rem}

\begin{prp}\label{PRP.QE_QForm_I}
For all $\mu,\eta\in A_{+}^{*}$ and $\varepsilon>0$, we have

\begin{itemize}
\item[1)] positive bounded quadratic form $\mathbf{L}_{\mu,\varepsilon}^{-\phi}$ on $L^{2}(B,\omega)$ s.t.~

\begin{itemize}
\item[1.1)] its representing operator $L_{\mu,\varepsilon}^{-\phi}\in\BII\lc{}L^{2}(B,\omega)\rc_{+}$ is injective,

\item[1.2)] $0\leq L_{\mu,\varepsilon}^{-\phi}=\s$-$\lim_{j\in\mathbb{N}}\pi_{j}^{B}\big(L_{\phi\lc\mu_{j}\rc{}}+\varepsilon I\big)^{-1}\pi_{j}^{B}\leq\varepsilon^{-1} I$,
\end{itemize}

\begin{reapply}
\end{reapply}

\item[2)] positive bounded quadratic form $\mathbf{R}_{\eta,\varepsilon}^{-\bpsi}$ on $L^{2}(B,\omega)$ s.t.~

\begin{itemize}
\item[2.1)] its representing operator $R_{\eta,\varepsilon}^{-\bpsi}\in\BII\lc{}L^{2}(B,\omega)\rc_{+}$ is injective,

\item[2.2)] $0\leq R_{\eta,\varepsilon}^{-\bpsi}=\s$-$\lim_{j\in\mathbb{N}}\pi_{j}^{B}\big(R_{\bpsi\lc\eta_{j}\rc{}}+\varepsilon I\big)^{-1}\pi_{j}^{B}\leq\varepsilon^{-1} I$. 
\end{itemize}

\begin{reapply}
\end{reapply}

\end{itemize}
\end{prp}
\begin{proof}
Let $\mu,\eta\in A_{+}^{*}$ and $\varepsilon>0$. Equation \ref{EQ.SSEC.NCDS_NCD_Operators_5} and Equation \ref{EQ.SSEC.NCDS_NCD_Operators_7} show $\mathbf{L}_{\mu,\varepsilon}^{-\phi}$ and $\mathbf{R}_{\eta,\varepsilon}^{-\bpsi}$ are positive bounded quadratic forms on $L^{2}(B,\omega)$. The Kato-Robinson theorem ensures the existence of the positive bounded representing operators.\par
For all $u\in L^{2}(B,\omega)$ and $j\in\mathbb{N}$, we have

\begin{align}\label{EQ.PRP.QE_QForm_I_1}
\lc\|\mu_{j}\|_{\infty}+\varepsilon\rc^{-1}\|u_{j}\|_{\omega}^{2}\leq\lgl\lc\phi\lc\mu_{j}\rc{}+\varepsilon I\rc^{-1}u_{j},u_{j}\rgl_{\omega}.    
\end{align}

\noindent Since $\|u\|_{\omega}=\sup_{j\in\mathbb{N}}\|u_{j}\|_{\omega}$ in each case, Equation \ref{EQ.PRP.QE_QForm_I_1} implies injectivity of $L_{\mu,\varepsilon}^{-\phi}$. Get $1.1)$. We know $1.2)$ by Equation \ref{EQ.SSEC.NCDS_NCD_Operators_4}. Altogether, $1)$ holds. We show $2)$ analogously.
\end{proof}

\begin{dfn}\label{DFN.QE_QForm_CLRA}
For all $\mu,\eta\in A_{+}^{*}$ and $\varepsilon>0$, we call

\begin{itemize}
\item[1)] the representing operator $L_{\mu,\varepsilon}^{-\phi}$ of $\mathbf{L}_{\mu,\varepsilon}^{-\phi}$ left-division by $\mu$ perturbed with $\varepsilon$, and $L_{\mu,\varepsilon}^{\phi}:=\big(L_{\mu,\varepsilon}^{-\phi}\big)^{-1}$ left-multiplication by $\mu$ perturbed with $\varepsilon$,

\item[2)] the representing operator $R_{\eta,\varepsilon}^{-\bpsi}$ of $\mathbf{R}_{\eta,\varepsilon}^{-\bpsi}$ right-division by $\eta$ perturbed with $\varepsilon$, and $R_{\eta,\varepsilon}^{\bpsi}:=\big(R_{\eta,\varepsilon}^{-\bpsi}\big)^{-1}$ right-multiplication by $\eta$ perturbed with $\varepsilon$.
\end{itemize}
\end{dfn}

\begin{ntn}\label{NTN.QE_QForm_CLRA}
We suppress $\phi$ and $\bpsi$ in Definition \ref{DFN.QE_QForm_CLRA} if $\phi=\bpsi=\id_{A}$.
\end{ntn}

\begin{rem}\label{REM.QE_QForm_CLRA}
For all $\mu,\eta\in A_{+}^{*}$, $\varepsilon>0$ and $j\in\mathbb{N}$, note $I=\s$-$\lim_{j\in\mathbb
{N}}\pi_{j}^{B}$ implies

\begin{align}\label{EQ.REM.QE_QForm_CLRA_1}
L_{\mu_{j},\varepsilon}^{-\phi}=\big(L_{\phi\lc\mu_{j}\rc{}}+\varepsilon I\big)^{-1},\ R_{\eta_{j},\varepsilon}^{-\bpsi}=\big(R_{\bpsi\lc\eta_{j}\rc{}}+\varepsilon I\big)^{-1}
\end{align}

\noindent and therefore $L_{\mu_{j},\varepsilon}^{\phi}=L_{\phi\lc\mu_{j}\rc{}}+\varepsilon I$, $R_{\eta_{j},\varepsilon}^{\bpsi}=R_{\bpsi\lc\eta_{j}\rc{}}+\varepsilon I$.
\end{rem}

\begin{prp}\label{PRP.QE_QForm_II}
For all $\mu,\eta\in A_{+}^{*}$ and $\varepsilon>0$, we have 

\begin{itemize}
\item[1)] $L_{\mu,\varepsilon}^{\phi},R_{\eta,\varepsilon}^{\bpsi}\in\UBII\lc{}L^{2}(B,\omega)\rc_{+}$ commute strongly and $L_{\mu,\varepsilon}^{\phi},R_{\eta,\varepsilon}^{\bpsi}\geq\varepsilon I$,

\item[2)] $L_{\mu,\varepsilon}^{\phi}=\sr$-$\lim_{j\in\mathbb{N}}L_{\mu_{j},\varepsilon}^{\phi}$ and $R_{\eta,\varepsilon}^{\bpsi}=\sr$-$\lim_{j\in\mathbb{N}}R_{\eta_{j},\varepsilon}^{\bpsi}$.
\end{itemize}
\end{prp}
\begin{proof}
We know $L_{\mu,\varepsilon}^{\phi},R_{\eta,\varepsilon}^{\bpsi}\in\UBII\lc{}L^{2}(B,\omega)\rc_{+}$ and $L_{\mu,\varepsilon}^{\phi},R_{\eta,\varepsilon}^{\bpsi}\geq\varepsilon I$ by Proposition \ref{PRP.QE_QForm_I} as inversion reverts partial order \lc{}cf.~Proposition \ref{PRP.Unbd_PO_Inversion}\rc{}. We show strong commutativity. Since we have uniform lower bound $\varepsilon>0$, resolvents in $a=0$ are respective perturbed division operators. Using sequential strong continuity of multiplication and the inverse of Equation \ref{EQ.REM.QE_QForm_CLRA_1}, we calculate

\begin{align}\label{EQ.PRP.QE_QForm_II_1}
L_{\mu,\varepsilon}^{-\phi}R_{\eta,\varepsilon}^{-\bpsi}=\s\textrm{-}\lim_{j\in\mathbb{N}}\hspace{0.025cm} L_{\mu_{j},\varepsilon}^{-\phi}R_{\eta_{j},\varepsilon}^{-\bpsi}=\s\textrm{-}\lim_{j\in\mathbb{N}}\hspace{0.025cm} R_{\eta_{j},\varepsilon}^{-\bpsi}L_{\mu_{j},\varepsilon}^{-\phi}=R_{\eta,\varepsilon}^{-\bpsi}L_{\mu,\varepsilon}^{-\phi}.
\end{align}

\noindent Equation \ref{EQ.PRP.QE_QForm_II_1} is commutativity of resolvents in $a=0$. Proposition 5.27 in \cite{BK.Sch.2012.Unbounded_Operators} then implies strong commutativity. Get $1)$. We have $2)$ by $1)$ in Proposition \ref{PRP.SR} for $a=0$.
\end{proof}

Definition \ref{DFN.QE_QForm_Perturbed_Division} uses bounded measurable joint functional calculus of strongly commuting self-adjoint unbounded operators \lc{}cf.~Definition \ref{DFN.JFC_Unbd}\rc{}. For details on spectral integration and the latter functional calculus, we refer to Subsection \ref{SSEC.A_Fnd_FC}.

\begin{dfn}\label{DFN.QE_QForm_Perturbed_Division}
For all $\mu,\eta\in A_{+}^{*}$ and $\varepsilon>0$, we call $\mathcal{D}_{\mu,\eta,\varepsilon}:=m_{f}^{-1}\lc{}L_{\mu,\varepsilon}^{\phi},R_{\eta,\varepsilon}^{\bpsi}\rc$ the noncommutative division operator of $\mu$ and $\eta$ perturbed with $\varepsilon$.
\end{dfn}

\begin{prp}\label{PRP.QE_QForm_Perturbed_Division}
Let $\mu,\eta\in A_{+}^{*}$.

\begin{itemize}
\item[1)] For all $\varepsilon>0$, we have 

\begin{itemize}
\item[1.1)] $\mathcal{D}_{\mu,\eta,\varepsilon}=\s$-$\lim_{j\in\mathbb{N}}\mathcal{D}_{\mu_{j},\eta_{j},\varepsilon}\in\BII\lc{}L^{2}(B,\omega)\rc_{+}$ and $\big\|\mathcal{D}_{\mu,\eta,\varepsilon}\big\|_{\BII\lc{}L^{2}(B,\omega)\rc{}}\leq\varepsilon^{-1}$,

\item[1.2)] $\mathcal{D}_{\mu_{j},\eta_{j},\varepsilon}=\mathcal{D}_{\mu_{j}+\varepsilon 1_{A},\eta_{j}+\varepsilon 1_{A}}$ for all $j\in\mathbb{N}$.
\end{itemize}

\begin{reapply}
\end{reapply}

\item[2)] We have monotonically increasing net $\lset\mathcal{D}_{\mu,\eta,\varepsilon}\rset_{\varepsilon>0}\subset\BII\lc{}L^{2}(B,\omega)\rc_{+}$ in dual order.
\end{itemize}
\end{prp}


\pagebreak


\begin{proof}
Using Proposition \ref{PRP.QE_QForm_II} and Remark \ref{REM.OM_Representing_Fct}, get $1.1)$ by Lemma \ref{LEM.FC_SR} applied to $m_{f}^{-1}$. Using Proposition \ref{PRP.AF_Cstar_Bimodule_CLRA_I}, get $1.2)$ by functional calculus upon taking inverses in Equation \ref{EQ.REM.QE_QForm_CLRA_1} since $\phi,\bpsi$ are unital. Altogether, get $1)$.\par
We show $2)$. For all $j\in\mathbb{N}$ and $\varepsilon_{1}\geq\varepsilon_{0}>0$ in $\mathbb{R}$, we use $2)$ in Proposition \ref{PRP.Quadratic_Form} and $3)$ in Theorem \ref{THM.QE_AF} to estimate

\begin{align}\label{EQ.PRP.QE_QForm_Perturbed_Division_1}
\mathbf{L}_{\mu_{j},\varepsilon_{1}}^{-\phi}\leq \mathbf{L}_{\mu_{j},\varepsilon_{0}}^{-\phi},\ \mathbf{R}_{\eta_{j},\varepsilon_{1}}^{-\bpsi}\leq \mathbf{R}_{\eta_{j},\varepsilon_{0}}^{-\bpsi}.
\end{align}

\noindent Using positivity-preservation of representing operators, Equation \ref{EQ.PRP.QE_QForm_Perturbed_Division_1} shows

\begin{align}\label{EQ.PRP.QE_QForm_Perturbed_Division_2}
L_{\mu_{j},\varepsilon_{1}}^{-\phi}\leq L_{\mu_{j},\varepsilon_{0}}^{-\phi},\ R_{\eta_{j},\varepsilon_{1}}^{-\bpsi}\leq R_{\eta_{j},\varepsilon_{0}}^{-\bpsi}.
\end{align}

\noindent Letting $j\uparrow\infty$ in Equation \ref{EQ.PRP.QE_QForm_Perturbed_Division_2} yields

\begin{align}\label{EQ.PRP.QE_QForm_Perturbed_Division_3}
L_{\mu,\varepsilon_{1}}^{-\phi}\leq L_{\mu,\varepsilon_{0}}^{-\phi},\ R_{\eta,\varepsilon_{1}}^{-\bpsi}\leq R_{\eta,\varepsilon_{0}}^{-\bpsi}
\end{align}

\noindent in strong limit. Since inversion reverts partial order \lc{}cf.~Proposition \ref{PRP.Unbd_PO_Inversion}\rc{}, taking the inverses in Equation \ref{EQ.PRP.QE_QForm_Perturbed_Division_3} shows

\begin{align}\label{EQ.PRP.QE_QForm_Perturbed_Division_4}
L_{\mu,\varepsilon_{0}}^{\phi}\leq L_{\mu,\varepsilon_{1}}^{\phi},\ R_{\eta,\varepsilon_{0}}^{\bpsi}\leq R_{\eta,\varepsilon_{1}}^{\bpsi}.
\end{align}

\noindent Using $1.2)$ and Proposition \ref{PRP.OM_Representing_Fct}, Equation \ref{EQ.PRP.QE_QForm_Perturbed_Division_4} implies $2)$ by functional calculus.
\end{proof}

\begin{lem}\label{LEM.QE_QForm_Representation}
For all $\mu,\eta\in A_{+}^{*}$ and $u\in L^{2}(B,\omega)$, we have

\begin{itemize}
\item[1)] $\mathcal{I}_{j}^{f,\theta}\lc\mu_{j},\eta_{j},u_{j}\rc{}=\sup_{\varepsilon>0}\lgl\mathcal{D}_{\mu_{j},\eta_{j},\varepsilon}^{\theta}(u_{j}),u_{j}\rgl_{\omega}$ for all $j\in\mathbb{N}$,

\item[2)] $\sup_{j\in\mathbb{N}}\sup_{\varepsilon>0}\lgl\mathcal{D}_{\mu_{j},\eta_{j},\varepsilon}^{\theta}(u_{j}),u_{j}\rgl_{\omega}=\sup_{\varepsilon>0}\sup_{j\in\mathbb{N}}\lgl\mathcal{D}_{\mu_{j},\eta_{j},\varepsilon}^{\theta}(u_{j}),u_{j}\rgl_{\omega}$,

\item[3)] $\sup_{j\in\mathbb{N}}\lgl\mathcal{D}_{\mu_{j},\eta_{j},\varepsilon}^{\theta}(u_{j}),u_{j}\rgl_{\omega}=\lgl\mathcal{D}_{\mu,\eta,\varepsilon}^{\theta}(u),u\rgl_{\omega}$ for all $\varepsilon>0$.
\end{itemize}
\end{lem}
\begin{proof}
Let $\mu,\eta\in A_{+}^{*}$ and $u\in L^{2}(B,\omega)$. We use Corollary \ref{COR.AF_Cstar_Bimodule_Compression_Restriction}. For all $\varepsilon>0$ and $j\in\mathbb{N}$, we see $1.2)$ in Proposition \ref{PRP.QE_QForm_Perturbed_Division} and $2)$ in Lemma \ref{LEM.NCD_Operator_Compressed_PMO_I} show

\begin{align}\label{EQ.LEM.QE_QForm_Representation_1}
\restr{0.925}{\mathcal{D}_{\mu_{j},\eta_{j},\varepsilon}^{\theta}}{B_{j}}=\restr{0.925}{\mathcal{D}_{\mu_{j}+\varepsilon 1_{A},\eta_{j}+\varepsilon 1_{A}}^{\theta}}{B_{j}}=\mathcal{D}_{\mu_{j},\eta_{j},B_{j},\varepsilon}^{\theta}.
\end{align}

\noindent Equation \ref{EQ.LEM.QE_QForm_Representation_1} shows $1)$ by construction of quasi-entropies. Note $\sup_{j\in J}\sup_{k\in K}a_{j,k}=\sup_{k\in K}\sup_{j\in J}a_{j,k}$ for all double-indexed real sequences. The latter shows $2)$ at once. For all $\varepsilon>0$, monotonicity of quasi-entropies shows

\begin{align}\label{EQ.LEM.QE_QForm_Representation_2}
\sup_{j\in\mathbb{N}}\hspace{0.025cm} \lgl\mathcal{D}_{\mu_{j},\eta_{j},\varepsilon}^{\theta}(u_{j}),u_{j}\rgl_{\omega}=\lim_{j\in\mathbb{N}}\hspace{0.025cm} \lgl\mathcal{D}_{\mu_{j},\eta_{j},\varepsilon}^{\theta}(u_{j}),u_{j}\rgl_{\omega}.
\end{align}

\noindent For all $j\in\mathbb{N}$, $u_{j}=\pi_{j}^{B}(u)$ . Thus $u=\|.\|_{\omega}$-$\lim_{j\in\mathbb{N}}u_{j}$, hence $3)$ follows by Equation \ref{EQ.LEM.QE_QForm_Representation_2} and $1.1)$ in Proposition \ref{PRP.QE_QForm_Perturbed_Division}. We use uniform boundedness in our calculation.
\end{proof}

\begin{thm}\label{THM.QE_QForm_Representation}
Let $(A,\tau)$ and $(B,\omega)$ be tracial AF-$C^{*}$-algebras. Let $(\phi,\bpsi,\gamma)$ be an AF-$A$-bimodule structure on $B$. Let $f$ be representing function of an operator mean and $\theta\in [0,1]$. For all $\mu,\eta\in A_{+}^{*}$, we have

\begin{itemize}
\item[1)] $Q_{\mu,\eta}^{f,\theta}:L^{2}(B,\omega)\longrightarrow [0,\infty]$ is closed positive unbounded quadratic form on $L^{2}(B,\omega)$ represented uniquely by the positive unbounded operator defined by

\begin{align}\label{EQ.THM.QE_QForm_Representation_1}
\mathcal{D}_{\mu,\eta}^{\theta}:=\sr\textrm{-}\lim_{\varepsilon\downarrow 0}\hspace{0.025cm} \mathcal{D}_{\mu,\eta,\varepsilon}^{\theta}
\end{align}

\begin{reapply}
\end{reapply}

\noindent on $H\big(Q_{\mu,\eta}^{f,\theta}\hspace{0.015cm} \big)$,

\item[2)] $Q_{\mu,\eta}^{f,\theta}(u)=\big\|\mathcal{D}_{\mu,\eta}^{\frac{\theta}{2}}(u)\big\|_{\omega}^{2}=\sup_{\varepsilon>0}\lgl\mathcal{D}_{\mu,\eta,\varepsilon}^{\theta}(u),u\rgl_{\omega}$ for all $u\in L^{2}(B,\omega)$.
\end{itemize}
\end{thm}
\begin{proof}
For all $u\in L^{2}(B,\omega)$, get $Q_{\mu,\eta}^{f,\theta}(u)=\sup_{j\in\mathbb{N}}\mathcal{I}_{j}^{f,\theta}\lc\mu_{j},\eta_{j},u_{j}\rc$ by definition. Consecutive application of $1)$ to $3)$ in Lemma \ref{LEM.QE_QForm_Representation} lets us calculate

\begin{align}\label{EQ.THM.QE_QForm_Representation_2}
Q_{\mu,\eta}^{f,\theta}(u)=\sup_{\varepsilon>0}\hspace{0.025cm} \sup_{j\in\mathbb{N}}\hspace{0.025cm} \lgl\mathcal{D}_{\mu_{j},\eta_{j},\varepsilon}^{\theta}(u_{j}),u_{j}\rgl_{\omega}=\sup_{\varepsilon>0}\hspace{0.025cm} \lgl\mathcal{D}_{\mu,\eta,\varepsilon}^{\theta}(u),u\rgl_{\omega}
\end{align}

\noindent for all $u\in L^{2}(B,\omega)$. Equation \ref{EQ.THM.QE_QForm_Representation_2} implies our claims by $2)$ in Proposition \ref{PRP.QE_QForm_Perturbed_Division} and the Kato-Robinson theorem \lc{}cf.~Theorem 10.4.2 in \cite{BK.deOli.2009.OpAlg_Quantum_Dynamics}\rc{}. Note Remark \ref{REM.QE_QForm_SR} for uniqueness of strong resolvent limits for Equation \ref{EQ.THM.QE_QForm_Representation_1}.
\end{proof}

\begin{dfn}\label{DFN.QE_QForm_Representation}
For all $\mu,\eta\in A_{+}^{*}$, we call $\mathcal{D}_{\mu,\eta}^{\theta}$ in Equation \ref{EQ.THM.QE_QForm_Representation_1} the noncommutative division operator of $\mu$ and $\eta$.
\end{dfn}


\subsubsection*{Noncommutative division operators in the normal case}

Definition \ref{DFN.NCD_Operator_Compressed_PMO} and Definition \ref{DFN.QE_QForm_Representation} are a priori different definitions of noncommutative division in the AF-$C^{*}$-setting. Using results in Theorem \ref{THM.AF_Cstar_Bimodule_CLRA_SR} and assuming the representing function induces operator mean vanishing on $[0,\infty)\times\lset{}0\rset\cup\lset{}0\rset\times [0,\infty)$, Theorem \ref{THM.NCD_Operator_Compressed_PMO} implies Definition \ref{DFN.QE_QForm_Representation} reduces to Definition \ref{DFN.NCD_Operator_Compressed_PMO} if and only if operator means have finite inverses w.r.t.~compressed joint spectral measures. Since we do not suppress the flat operator for positive integrable measurable operators, we distinguish inverses of canonical left-~and right-actions from perturbed noncommutative left-~and right-division.\par
Let $(A,\tau)$ and $(B,\omega)$ be tracial AF-$C^{*}$-algebras. Let $(\phi,\bpsi,\gamma)$ be an AF-$A$-bimodule structure on $B$. Let $f$ be representing function of an operator mean and $\theta\in [0,1]$.

\begin{prp}\label{PRP.AF_Cstar_Bimodule_L2Inf_SR}
Let $p\in\lset{}2,\infty\rset$ and $x\in L^{p}(A,\tau)_{h}$.

\begin{itemize}
\item[1)] For all $j\in\mathbb{N}$, we have $L_{x_{j}},R_{x_{j}}\in\BII\lc{}L^{2}(A,\tau)\rc_{h}\cap\UBII_{A_{j}}\lc{}L^{2}(A,\tau)\rc$ and

\begin{align}\label{EQ.PRP.AF_Cstar_Bimodule_L2Inf_SR_1}
\pi_{j}^{A}L_{x_{j}}=\comAj L_{x},\ \pi_{j}^{A}R_{x_{j}}=\comAj R_{x}.
\end{align}

\begin{reapply}
\end{reapply}

\item[2)] $L_{x}=\sr$-$\lim_{j\in\mathbb{N}}L_{x_{j}}$ and $R_{x}=\sr$-$\lim_{j\in\mathbb{N}}R_{x_{j}}$.
\end{itemize}
\end{prp}


\pagebreak


\begin{proof}
For all $T\in\BII\lc{}L^{2}(A,\tau)\rc$, get $\comAj T=\pi_{j}^{A}T\pi_{j}^{A}$ \lc{}cf.~Definition \ref{DFN.Compression_Concrete}\rc{}. We prove all claims for canonical left-action. This readily transfers to canonical right-actions. For all $j\in\mathbb{N}$, we directly verify the identity of bounded operators

\begin{align}\label{EQ.PRP.AF_Cstar_Bimodule_L2Inf_SR_2}
\pi_{j}^{A}L_{x_{j}}=\comAj L_{x}    
\end{align}

\noindent on inner products. The above calculation uses $A_{j}$ is a $^{*}$-algebra. If $x$ is self-adjoint, then Equation \ref{EQ.PRP.AF_Cstar_Bimodule_L2Inf_SR_2} implies $A_{j}$-reducibility. Get $1)$. We show $2)$. Assume $x$ is self-adjoint. If $p=2$, then $2)$ in Proposition \ref{PRP.SR} for core $L^{2,\infty}(A,\tau)$ and Corollary \ref{COR.Wstar_CLRA_V} show our claim. If $p=\infty$, then $2)$ in Proposition \ref{PRP.SR} for core $L^{2}(A,\tau)$ and boundedness do.
\end{proof}

\begin{lem}\label{LEM.AF_Cstar_Bimodule_L1_SR}
Let $x\in L^{1}(A,\tau)_{+}$.

\begin{itemize}
\item[1)] $L_{x}=\sr$-$\lim_{n\in\mathbb{N}}L_{\min\{x,n\}}$ and $R_{x}=\sr$-$\lim_{n\in\mathbb{N}}R_{\min\{x,n\}}$.

\item[2)] For all $j\in\mathbb{N}$, we have

\begin{align}\label{EQ.LEM.AF_Cstar_Bimodule_L1_SR_1}
\comAj L_{x_{j}}\leq L_{\pi_{j}^{A}(\sqrt{x})}^{2},\ \comAj R_{x_{j}}\leq R_{\pi_{j}^{A}(\sqrt{x})}^{2}.
\end{align}

\begin{reapply}
\end{reapply}

\item[3)] For all $\varepsilon>0$, we have

\begin{align}\label{EQ.LEM.AF_Cstar_Bimodule_L1_SR_2}
R_{-\varepsilon}(L_{x})\leq L_{x^{\flat},\varepsilon}^{-\id_{A}},\ R_{-\varepsilon}(R_{x})\leq R_{x^{\flat},\varepsilon}^{-\id_{A}}.    
\end{align}

\begin{reapply}
\end{reapply}

\end{itemize}
\end{lem}
\begin{proof}
We prove all claims for canonical left-action. This readily transfers to canonical right-actions. By $2)$ in Lemma \ref{LEM.Wstar_CLRA_FC} and functional calculus, we have monotonically increasing $\lset{}L_{\min\{x,n\}}\rset_{n\in\mathbb{N}}=\lset\hspace{-0.0375cm} \min\{L_{x},n\}\rset_{n\in\mathbb{N}}\subset\BII\lc{}L^{2}(A,\tau)\rc_{+}$. Applying the Kato-Robinson theorem, we directly verify $1)$ on closed positive unbounded quadratic forms.\par
We show $2)$. Let $j\in\mathbb{N}$. We know $\sqrt{x}_{j}=\pi_{j}^{A}(\sqrt{x})$. Using $1)$ in Proposition \ref{PRP.AF_Cstar_Bimodule_L2Inf_SR} and $1.3)$ in Proposition \ref{PRP.Reducible}, we have the identity of bounded operators

\begin{align}\label{EQ.LEM.AF_Cstar_Bimodule_L1_SR_3}
L_{\pi_{j}^{A}(\sqrt{x})}=\comAj L_{\sqrt{x}}+\big(I-\pi_{j}^{A}\big)L_{\pi_{j}^{A}(\sqrt{x})}\big(I-\pi_{j}^{A}\big).
\end{align}

\noindent Multiplying out terms as per Equation \ref{EQ.LEM.AF_Cstar_Bimodule_L1_SR_3} lets us estimate

\begin{align}\label{EQ.LEM.AF_Cstar_Bimodule_L1_SR_4}
L_{\pi_{j}^{A}(\sqrt{x})}^{2}=\lc\comAj L_{\sqrt{x}}\rc^{2}+\lc\big(I-\pi_{j}^{A}\big)L_{\pi_{j}^{A}(\sqrt{x})}\big(I-\pi_{j}^{A}\big)\rc^{2}\geq\lc\comAj L_{\sqrt{x}}\rc^{2}.
\end{align}

\noindent Furthermore, we calculate

\begin{align*}
\lgl\lc\comAj L_{\sqrt{x}}\rc^{2}(u),u\rgl_{\tau} & =\tau\big(x\pi_{j}^{A}(u)\pi_{j}^{A}(u)^{*}\big)+\lgl\big(I-\pi_{j}^{A}\big)\lc{}L_{\sqrt{x}}\pi_{j}^{A}(u)\rc{},L_{\sqrt{x}}\pi_{j}^{A}(u)\rgl_{\tau} \phantom{\bigg)} \\
& \geq\tau\big(x\pi_{j}^{A}(u)\pi_{j}^{A}(u)^{*}\big)=\lgl\comAj L_{x_{j}}(u),u\rgl_{\tau}\phantom{\bigg)}
\end{align*}

\noindent for all $u\in L^{2}(A,\tau)$.\par


\pagebreak


The above calculation implies

\begin{align}\label{EQ.LEM.AF_Cstar_Bimodule_L1_SR_5}
\lc\comAj L_{\sqrt{x}}\rc^{2}\geq\comAj L_{x_{j}}.
\end{align}

\noindent Equation \ref{EQ.LEM.AF_Cstar_Bimodule_L1_SR_4} and Equation \ref{EQ.LEM.AF_Cstar_Bimodule_L1_SR_5} show $2)$.\par
We show $3)$. Let $\varepsilon>0$. For all $j\in\mathbb{N}$, Equation \ref{EQ.LEM.AF_Cstar_Bimodule_L1_SR_1} yields

\begin{align}\label{EQ.LEM.AF_Cstar_Bimodule_L1_SR_6}
R_{-\varepsilon}\bigg(L_{\pi_{j}^{A}(\sqrt{x})}^{2}\bigg)\leq R_{-\varepsilon}\lc\pi_{j}^{A}L_{x_{j}}\pi_{j}^{A}\rc{}
\end{align}

\noindent since inversion reverts partial order \lc{}cf.~Proposition \ref{PRP.Unbd_PO_Inversion}\rc{}. Let $j\in\mathbb{N}$. Then using $2)$ in Lemma \ref{LEM.Compression_Preservation_I}, we directly verify

\begin{align}\label{EQ.LEM.AF_Cstar_Bimodule_L1_SR_7}
R_{-\varepsilon}\lc\comAj L_{x_{j}}\rc{}=\comAj\lc\comAj L_{x_{j}}+\varepsilon\pi_{j}^{A}\rc^{-1}+\varepsilon^{-1}\big(I-\pi_{j}^{A}\big),
\end{align}

\noindent and

\begin{align}\label{EQ.LEM.AF_Cstar_Bimodule_L1_SR_8}
\comAj L_{x_{j}^{\flat},\varepsilon}^{-\id_{A}}=\comAj\lc\comAj L_{x_{j}}+\varepsilon\pi_{j}^{A}\rc^{-1}.
\end{align}

\noindent Equation \ref{EQ.LEM.AF_Cstar_Bimodule_L1_SR_6}, Equation \ref{EQ.LEM.AF_Cstar_Bimodule_L1_SR_7} and Equation \ref{EQ.LEM.AF_Cstar_Bimodule_L1_SR_8} let us estimate

\begin{align}\label{EQ.LEM.AF_Cstar_Bimodule_L1_SR_9}
R_{-\varepsilon}\bigg(L_{\pi_{j}^{A}(\sqrt{x})}^{2}\bigg) & \leq R_{-\varepsilon}\lc\comAj L_{x_{j}}\rc{}=\comAj L_{x_{j}^{\flat},\varepsilon}^{-\id_{A}}+\varepsilon^{-1}\big(I-\pi_{j}^{A}\big).
\end{align}

Note $I=\s$-$\lim_{j\in\mathbb{N}}\pi_{j}^{A}$ is uniformly bounded in norm \lc{}cf.~$3)$ in Proposition \ref{PRP.AF_Cstar_Trace_III}\rc{}. By construction of perturbed left-division, sequential strong continuity of multiplication ensures

\begin{align}\label{EQ.LEM.AF_Cstar_Bimodule_L1_SR_10}
L_{x^{\flat},\varepsilon}^{-\id_{A}}=\s\textrm{-}\lim_{j\in\mathbb{N}}\hspace{0.025cm} \comAj L_{x_{j}^{\flat},\varepsilon}^{-\id_{A}}+\varepsilon^{-1}\big(I-\pi_{j}^{A}\big).
\end{align}

\noindent The map $t\mapsto R_{-\varepsilon}\lc{}t^{2}\rc$ lies in $C_{b}(\mathbb{R})$. Using Lemma \ref{LEM.FC_SR}, we see $2)$ in Proposition \ref{PRP.AF_Cstar_Bimodule_L2Inf_SR} thus implies

\begin{align}\label{EQ.LEM.AF_Cstar_Bimodule_L1_SR_11}
R_{-\varepsilon}(L_{x})=R_{-\varepsilon}\lc{}L_{\sqrt{x}}^{2}\rc{}=\s\textrm{-}\lim_{j\in\mathbb{N}}\hspace{0.025cm} R_{-\varepsilon}\bigg(L_{\pi_{j}^{A}(\sqrt{x})}^{2}\bigg)
\end{align}

\noindent by functional calculus.\par


\pagebreak


Using Equation \ref{EQ.LEM.AF_Cstar_Bimodule_L1_SR_9}, Equation \ref{EQ.LEM.AF_Cstar_Bimodule_L1_SR_10} and Equation \ref{EQ.LEM.AF_Cstar_Bimodule_L1_SR_11}, we calculate

\begin{align*}
R_{-\varepsilon}(L_{x}) & = \s\textrm{-}\lim_{j\in\mathbb{N}}\hspace{0.025cm} R_{-\varepsilon}\bigg(L_{\pi_{j}^{A}(\sqrt{x})}^{2}\bigg) \phantom{\Bigg)} \\
& \leq \s\textrm{-}\lim_{j\in\mathbb{N}}\hspace{0.025cm} \comAj L_{x_{j}^{\flat},\varepsilon}^{-\id_{A}}+\varepsilon^{-1}\big(I-\pi_{j}^{A}\big) \phantom{\Bigg)} \\
& = L_{x^{\flat},\varepsilon}^{-\id_{A}}. \phantom{\Bigg)}
\end{align*}

\noindent The above calculation shows $3)$ at once.
\end{proof}

\begin{thm}\label{THM.AF_Cstar_Bimodule_CLRA_SR}
Let $(A,\tau)$ and $(B,\omega)$ be tracial AF-$C^{*}$-algebras. Let $(\phi,\bpsi,\gamma)$ be an AF-$A$-bimodule structure on $B$. Let $f$ be representing function of an operator mean and $\theta\in [0,1]$. Let $p\in\lset{}1,2,\infty\rset$. For all $x\in L^{p}(A,\tau)_{+}$, we have

\begin{align}\label{EQ.THM.AF_Cstar_Bimodule_CLRA_SR_1}
L_{x}^{\phi}=\sr\textrm{-}\lim_{j\in\mathbb{N}}\hspace{0.025cm} L_{x_{j}}^{\phi},\ R_{x}^{\bpsi}=\sr\textrm{-}\lim_{j\in\mathbb{N}}\hspace{0.025cm} R_{x_{j}}^{\bpsi}.
\end{align}
\end{thm}
\begin{proof}
We prove all claims for canonical left-action. This readily transfers to canonical right-actions. First, we show our claim for canonical AF-$C^{*}$-bimodules. Secondly, we extend to the general case. In this proof, $\gamma$ is of no consequence. Let $x\in L^{p}(A,\tau)_{+}$.\par
Assume $(A,\tau)=(B,\omega)$ is equipped with its canonical AF-$A$-bimodule structure. By $1)$ in Proposition \ref{PRP.SR}, $R_{-\varepsilon}(L_{x})=\s$-$\lim_{j\in\mathbb{N}}R_{-\varepsilon}\lc{}L_{x_{j}}\rc$ for fixed but arbitrary $\varepsilon>0$ implies $L_{x}=\sr$-$\lim_{j\in\mathbb{N}}L_{x_{j}}$. Let $\varepsilon>0$. Using $1.2)$ in Proposition \ref{PRP.QE_QForm_I}, uniform boundedness and sequential strong continuity of multiplication, we calculate

\begin{align}\label{EQ.THM.AF_Cstar_Bimodule_CLRA_SR_2}
L_{x^{\flat},\varepsilon}^{-\id_{A}}=\s\textrm{-}\lim_{j\in\mathbb{N}}\hspace{0.025cm} \comAj\lc{}L_{x_{j}}+\varepsilon I\rc^{-1}=\s\textrm{-}\lim_{j\in\mathbb{N}}\hspace{0.025cm} \lc{}L_{x_{j}}+\varepsilon I\rc^{-1}=\s\textrm{-}\lim_{j\in\mathbb{N}}\hspace{0.025cm} R_{-\varepsilon}\lc{}L_{x_{j}}\rc{}.
\end{align}

\noindent It suffices to have $R_{-\varepsilon}(L_{x})=L_{x^{\flat},\varepsilon}^{-\id_{A}}$, ergo 

\begin{align}\label{EQ.THM.AF_Cstar_Bimodule_CLRA_SR_Fluff}
R_{-\varepsilon}(L_{x})\geq L_{x^{\flat},\varepsilon}^{-\id_{A}}    
\end{align}

\noindent by $3)$ in Lemma \ref{LEM.AF_Cstar_Bimodule_L1_SR}.\par
We use the following. For all $u\in L^{2}(A,\tau)$, get $w^{*}$-l.s.c.~map

\begin{align}\label{EQ.THM.AF_Cstar_Bimodule_CLRA_SR_3}
\mu\mapsto\mathbf{L}_{\mu,\varepsilon}^{-\id_{A}}(u)=\sup_{j\in\mathbb{N}}\hspace{0.025cm} \mathbf{L}_{\mu_{j},\varepsilon}^{-\id_{A}}(u)
\end{align}

\noindent defined on $A_{+}^{*}$ by $1)$ in Proposition \ref{PRP.Quadratic_Form}. Let $y\in L^{1,\infty}(A,\tau)_{+}$. Using $\sup_{j\in\mathbb{N}}\| y_{j}\|_{\infty}=\| y\|_{\infty}<\infty$, we estimate

\begin{align}\label{EQ.THM.AF_Cstar_Bimodule_CLRA_SR_4}
L_{y^{\flat},\varepsilon}^{-\id_{A}},L_{y_{j}^{\flat},\varepsilon}^{-\id_{A}}\geq\big(\| y\|_{\infty}+\varepsilon\big)^{-1}I>0
\end{align}

\noindent in $\BII\lc{}L^{2}(A,\tau)\rc$ for all $j\in\mathbb{N}$.\par


\pagebreak


Taking inverses in Equation \ref{EQ.THM.AF_Cstar_Bimodule_CLRA_SR_4} yields uniform bound s.t.~$2)$ in Proposition \ref{PRP.QE_QForm_II} implies

\begin{align}\label{EQ.THM.AF_Cstar_Bimodule_CLRA_SR_5}
L_{y^{\flat},\varepsilon}^{\id_{A}}=\s\textrm{-}\lim_{j\in\mathbb{N}}\hspace{0.025cm} L_{y_{j}^{\flat},\varepsilon}^{\id_{A}}.    
\end{align}

\noindent Following Remark \ref{EQ.REM.QE_QForm_CLRA_1}, Equation \ref{EQ.THM.AF_Cstar_Bimodule_CLRA_SR_5} shows

\begin{align}\label{EQ.THM.AF_Cstar_Bimodule_CLRA_SR_6}
L_{y^{\flat},\varepsilon}^{\id_{A}}=\s\textrm{-}\lim_{j\in\mathbb{N}}\hspace{0.025cm} L_{y_{j}}+\varepsilon I.      
\end{align}

\noindent Normality implies $L_{y}=\w$-$\lim_{j\in\mathbb{N}}L_{y_{j}}$ by $2)$ in Proposition \ref{PRP.AF_Cstar_Trace_Dualisation_II} \lc{}cf.~Remark \ref{REM.Wstar_Con} and Remark \ref{REM.Wstar_Normal_I}\rc{}. Therefore, Equation \ref{EQ.THM.AF_Cstar_Bimodule_CLRA_SR_6} lets us calculate

\begin{align}\label{EQ.THM.AF_Cstar_Bimodule_CLRA_SR_7}
L_{y^{\flat},\varepsilon}^{\id_{A}}=\s\textrm{-}\lim_{j\in\mathbb{N}}\hspace{0.025cm} L_{y_{j}}+\varepsilon I=L_{y}+\varepsilon I.     
\end{align}

\noindent Taking inverses in Equation \ref{EQ.THM.AF_Cstar_Bimodule_CLRA_SR_7}, we have $L_{y^{\flat},\varepsilon}^{-\id_{A}}=R_{-\varepsilon}(L_{y})$ and therefore

\begin{align}\label{EQ.THM.AF_Cstar_Bimodule_CLRA_SR_8}
\lgl R_{-\varepsilon}(L_{y})(u),u\rgl_{\tau}=\mathbf{L}_{y^{\flat},\varepsilon}^{-\id_{A}}(u)
\end{align}

\noindent for all $u\in L^{2}(A,\tau)$.\par
For all $n\in\mathbb{N}$, get $x_{n}:=\min\{x,n\}\in L^{1,\infty}(A,\tau)_{+}$ by positivity, as well as

\begin{align}\label{EQ.THM.AF_Cstar_Bimodule_CLRA_SR_9}
\lgl R_{-\varepsilon}\lc{}L_{x_{n}}\rc{}(u),u\rgl_{\tau}=\mathbf{L}_{x_{n}^{\flat},\varepsilon}^{-\id_{A}}(u)
\end{align}

\noindent for all $u\in L^{2}(A,\tau)$ by Equation \ref{EQ.THM.AF_Cstar_Bimodule_CLRA_SR_8}. Then $1)$ in Lemma \ref{LEM.AF_Cstar_Bimodule_L1_SR} and Lemma \ref{LEM.FC_SR} show

\begin{align}\label{EQ.THM.AF_Cstar_Bimodule_CLRA_SR_10}
R_{-\varepsilon}(L_{x})=\s\textrm{-}\lim_{n\in\mathbb{N}}\hspace{0.025cm} R_{-\varepsilon}\lc{}L_{x_{n}}\rc{}.
\end{align}

\noindent Equation \ref{EQ.THM.AF_Cstar_Bimodule_CLRA_SR_9} and Equation \ref{EQ.THM.AF_Cstar_Bimodule_CLRA_SR_10} let us calculate

\begin{align}\label{EQ.THM.AF_Cstar_Bimodule_CLRA_SR_11}
\lgl R_{-\varepsilon}(L_{x})(u),u\rgl_{\tau}=\lim_{n\in\mathbb{N}}\hspace{0.025cm} \lgl R_{-\varepsilon}\lc{}L_{x_{n}}\rc{}(u),u\rgl_{\tau}=\lim_{n\in\mathbb{N}}\hspace{0.025cm} \mathbf{L}_{x_{n}^{\flat}}^{-\id_{A},\varepsilon}(u)
\end{align}

\noindent for all $u\in L^{2}(A,\tau)$. Finally, note $x^{\flat}=w^{*}$-$\lim_{n\in\mathbb{N}}x_{n}^{\flat}$ by $2.2)$ in Proposition \ref{PRP.AF_Cstar_Trace_Dualisation_II}. Since the map in Equation \ref{EQ.THM.AF_Cstar_Bimodule_CLRA_SR_3} is $w^{*}$-l.s.c.~, Equation \ref{EQ.THM.AF_Cstar_Bimodule_CLRA_SR_11} shows

\begin{align}\label{EQ.THM.AF_Cstar_Bimodule_CLRA_SR_12}
\lgl R_{-\varepsilon}(L_{x})(u),u\rgl_{\tau}=\liminf_{n\in\mathbb{N}}\hspace{0.0675cm} \mathbf{L}_{x_{n}^{\flat}}^{-\id_{A},\varepsilon}(u)\geq\mathbf{L}_{x^{\flat}}^{-\id_{A}}(u)
\end{align}

\noindent for all $u\in L^{2}(A,\tau)$. Equation \ref{EQ.THM.AF_Cstar_Bimodule_CLRA_SR_12} implies Equation \ref{EQ.THM.AF_Cstar_Bimodule_CLRA_SR_Fluff}.\par


\pagebreak


Thus our claim holds assuming canonical AF-$C^{*}$-bimodule structure. Assume the general case. The map $t\mapsto R_{\pm i}(t)$ lies in $C_{b}(\mathbb{R})$. Applying $2)$ in Lemma \ref{LEM.Wstar_CLRA_FC} and using our above discussion, Lemma \ref{LEM.FC_SR} implies

\begin{align}\label{EQ.THM.AF_Cstar_Bimodule_CLRA_SR_13}
\Gamma_{x,L^{\infty}(A,\tau)}\lc{}R_{\pm i}\rc{}=\s\textrm{-}\lim_{n\in\mathbb{N}}\hspace{0.025cm} \Gamma_{x_{j},L^{\infty}(A,\tau)}\lc{}R_{\pm i}\rc{}.
\end{align}

\noindent For all $y\in L^{0}(A,\tau)_{h}$, Lemma \ref{LEM.AF_Cstar_Bimodule_Compression_I} shows

\begin{align}\label{EQ.THM.AF_Cstar_Bimodule_CLRA_SR_14}
L^{\phi}\lc\Gamma_{y,L^{\infty}(A,\tau)}\lc{}R_{\pm i}\rc\rc{}=R_{\pm i}\big(L_{y}^{\phi}\big).
\end{align}

\noindent We know $L^{\phi}:L^{\infty}(A,\tau)\longrightarrow\BII\lc{}L^{2}(B,\omega)\rc$ is normal unital $^{*}$-homomorphism. Thus $L^{\phi}$ is strongly continuous, hence Equation \ref{EQ.THM.AF_Cstar_Bimodule_CLRA_SR_13} and Equation \ref{EQ.THM.AF_Cstar_Bimodule_CLRA_SR_14} show

\begin{align*}
R_{\pm i}\big(L_{x}^{\phi}\big) & = L^{\phi}\lc\Gamma_{x,L^{\infty}(A,\tau)}\lc{}R_{\pm i}\rc\rc \phantom{\Bigg)} \\
& = \s\textrm{-}\lim_{n\in\mathbb{N}}\hspace{0.025cm} L^{\phi}\lc\Gamma_{x_{j},L^{\infty}(A,\tau)}\lc{}R_{\pm i}\rc\rc \phantom{\Bigg)} \\
& = \s\textrm{-}\lim_{n\in\mathbb{N}}\hspace{0.025cm} R_{\pm i}\big(L_{x_{j}}^{\phi}\big). \phantom{\Bigg)}
\end{align*}

\noindent The above calculation shows our claim.
\end{proof}

\begin{cor}\label{COR.AF_Cstar_Bimodule_CLRA_SR}
For all $x,y\in L^{1}(A,\tau)_{+}$ and $\varepsilon>0$, we have

\begin{itemize}
\item[1)] $L_{x^{\flat},\varepsilon}^{\phi}=L_{x}^{\phi}+\varepsilon I$ and $R_{y^{\flat},\varepsilon}^{\bpsi}=R_{y}^{\bpsi}+\varepsilon I$,

\item[2)] $\mathcal{D}_{x^{\flat},y^{\flat},\varepsilon}^{\theta}=\mathcal{D}_{x,y,\varepsilon}^{\theta}$.
\end{itemize}
\end{cor}
\begin{proof}
Let $x,y\in L^{1}(A,\tau)_{+}$ and $\varepsilon>0$. Equation \ref{EQ.REM.QE_QForm_CLRA_1} in Remark \ref{REM.QE_QForm_CLRA} rewrites as

\begin{align}\label{EQ.COR.AF_Cstar_Bimodule_CLRA_SR_1}
L_{x_{j}^{\flat},\varepsilon}^{-\phi}=\big(L_{x_{j}}^{\phi}+\varepsilon I\big)^{-1},\ R_{y_{j}^{\flat},\varepsilon}^{-\bpsi}=\big(R_{y_{j}}^{\bpsi}+\varepsilon I\big)^{-1}.
\end{align}

\noindent Using Equation \ref{EQ.PRP.NCD_Operator_Compressed_PMO_I_2} to ensure perturbation tends to zero as required, get $1)$ by applying Theorem \ref{THM.AF_Cstar_Bimodule_CLRA_SR} to inverses in Equation \ref{EQ.COR.AF_Cstar_Bimodule_CLRA_SR_1}. We then get $2)$ by $1)$ in Proposition \ref{PRP.NCD_Operator_Compressed_PMO_II} applied to the trivial compression.
\end{proof}

\begin{dfn}
We say that $f$ vanishes at the boundary if $m_{f}(\lambda,0)=0$ for all $\lambda\geq 0$.
\end{dfn}

\begin{rem}
If $f$ vanishes at the boundary, then $m_{f}(\lambda,0)=m_{f}(0,\lambda)=0$ for all $\lambda\geq 0$ by symmetry. Both the geometric and logarithmic operator means have representing function vanishing at the boundary. The arithmetic operator mean does not.
\end{rem}

\begin{lem}\label{LEM.NCD_Operator_Compressed_PMO_II}
Let $f$ vanish at the boundary.

\begin{itemize}
\item[1)] Let $p\in L^{\infty}(A,\tau)$ be a projection. For all $x,y\in L^{1}(A[p],\tau)_{+}$ and $\varepsilon>0$, we have

\begin{align}\label{EQ.LEM.NCD_Operator_Compressed_PMO_II_1}
\lb\mathcal{D}_{x,0,\varepsilon}^{\theta},L_{p}^{\phi}\rb{}=\lb\mathcal{D}_{0,y,\varepsilon}^{\theta},R_{p}^{\bpsi}\rb{}=0.    
\end{align}

\begin{reapply}
\end{reapply}

\item[2)] For all $x,y\in L^{1}(A,\tau)_{+}$ and $u\in L^{2}(B,\omega)$, we have 

\begin{align}\label{EQ.LEM.NCD_Operator_Compressed_PMO_II_2}
\sup_{\varepsilon>0}\hspace{0.025cm} \lgl\mathcal{D}_{x,0,\varepsilon}^{\theta}(u),u\rgl_{\omega}=\sup_{\varepsilon>0}\hspace{0.025cm} \lgl\mathcal{D}_{0,y,\varepsilon}^{\theta}(u),u\rgl_{\omega}=\infty. 
\end{align}

\begin{reapply}
\end{reapply}

\end{itemize}
\end{lem}
\begin{proof}
Let $x,y\in L^{1}(A,\tau)_{+}$. We prove all claims for $x$. Their proof readily transfers to the analogous one for $y$. We show $1)$. Let $p\in L^{\infty}(A,\tau)$ be a projection s.t.~$x\in L^{1}(A[p],\tau)_{+}$ and $\varepsilon>0$. By bounded measurable joint functional calculus \lc{}cf.~Proposition \ref{PRP.JFC_Bd}\rc{}, we have

\begin{align}\label{EQ.LEM.NCD_Operator_Compressed_PMO_II_3}
\mathcal{D}_{x,0,\varepsilon}^{\theta}=m_{f,\varepsilon}^{-\theta}\lc{}L_{x}^{\phi},0\rc\in W^{*}\big(L_{x}^{\phi},I\big).
\end{align}

\noindent Note $W^{*}\lc{}L_{x}^{\phi},I\rc{}=W^{*}(L_{x}^{\phi})\otimes W^{*}(I)\cong W^{*}(L_{x}^{\phi})\subset\BII\lc{}L^{2}(B,\omega)$ following Equation \ref{EQ.SSEC.A_Fnd_FC_8} in our construction of bounded measurable joint functional calculus. Since $x\in L^{1}(A[p],\tau)$, we know $\Gamma_{x,L^{\infty}(A,\tau)}\lc{}R_{\pm i}\rc\in W_{L^{\infty}(A,\tau)}^{*}(x)$.\par
Corollary \ref{COR.Wstar_Compression_Preservation_II} therefore implies

\begin{align}\label{EQ.LEM.NCD_Operator_Compressed_PMO_II_4}
\lb\Gamma_{x,L^{\infty}(A,\tau)}\lc{}R_{\pm i}\rc{},p\rb{}=\lb\Gamma_{x,L^{\infty}(A[p],\tau)}\lc{}R_{\pm i}-\mp i\rc{}+\mp i1_{A},p\rb{}=0.
\end{align}

\noindent Using Lemma \ref{LEM.AF_Cstar_Bimodule_Compression_I}, $2)$ in Lemma \ref{LEM.Wstar_CLRA_FC} and Equation \ref{EQ.LEM.NCD_Operator_Compressed_PMO_II_4}, we calculate

\begin{align}\label{EQ.LEM.NCD_Operator_Compressed_PMO_II_5}
\lb{}R_{\pm i}\big(L_{x}^{\phi}\big),L_{p}^{\phi}\rb{}=\lb{}L^{\phi}\lc\Gamma_{x,L^{\infty}(A,\tau)}\lc{}R_{\pm i}\rc\rc{},L_{p}^{\phi}\rb{}=L^{\phi}\lc\lb\Gamma_{x,L^{\infty}(A,\tau)}\lc{}R_{\pm i}\rc{},p\rb\rc{}=0.
\end{align}

\noindent Following Remark \ref{REM.JFC_Strong_Com}, Equation \ref{EQ.LEM.NCD_Operator_Compressed_PMO_II_5} shows

\begin{align}\label{EQ.LEM.NCD_Operator_Compressed_PMO_II_6}
\lb{}T,L_{p}^{\phi}\rb{}=0    
\end{align}

\noindent for all $T\in W^{*}(L_{x}^{\phi})$. Equation \ref{EQ.LEM.NCD_Operator_Compressed_PMO_II_3} and Equation \ref{EQ.LEM.NCD_Operator_Compressed_PMO_II_6} imply $1)$ at once.\par


\pagebreak


We show $2)$. Let $u\in L^{2}(B,\omega)$. Note $2)$ in Proposition \ref{PRP.NCD_Operator_Compressed_PMO_II} shows

\begin{align}\label{EQ.LEM.NCD_Operator_Compressed_PMO_II_7}
\sup_{\varepsilon>0}\hspace{0.025cm} \lgl\mathcal{D}_{x,0,\varepsilon}^{\theta}(u),u\rgl_{\omega}=\liminf_{\varepsilon\downarrow 0}\hspace{0.025cm} \lgl\mathcal{D}_{x,0,\varepsilon}^{\theta}(u),u\rgl_{\omega}.
\end{align}

\noindent Lemma \ref{LEM.Wstar_CLRA_FC} and Lemma \ref{LEM.Wstar_CLRA_JFC} imply

\begin{align}\label{EQ.LEM.NCD_Operator_Compressed_PMO_II_8}
\textrm{spec}_{L^{\infty}(A,\tau)}\hspace{0.055cm} x\times 0=\supp E_{L_{x}^{\phi},0}\times\vstretch{1.125}{\{}\hspace{0.007125cm} 0\hspace{0.007125cm}\vstretch{1.125}{\}}\subset\supp E_{L_{x}^{\phi}}\times\vstretch{1.125}{\{}\hspace{0.007125cm} 0\hspace{0.007125cm}\vstretch{1.125}{\}}=\textrm{spec}_{L^{\infty}(A,\tau)}\hspace{0.055cm} x\times 0.
\end{align}

\noindent For all $\varepsilon>0$, Equation \ref{EQ.LEM.NCD_Operator_Compressed_PMO_II_8} and Equation \ref{EQ.REM.SpecInt_Identities_1} show

\begin{align}\label{EQ.LEM.NCD_Operator_Compressed_PMO_II_9}
\lgl\mathcal{D}_{x,0,\varepsilon}^{\theta}(u),u\rgl_{\omega}=\int_{\spec_{L^{\infty}(A,\tau)}x\times 0}m_{f,\varepsilon}^{-\theta}\lc{}t,0\rc{}dE_{L_{x}^{\phi},0}^{u}.
\end{align}

Let $\lset\varepsilon_{n}\rset_{n\in\mathbb{N}}\subset (0,\infty)$ be a descending sequence converging to zero. Since $f$ vanishes at the boundary, we have

\begin{align}\label{EQ.LEM.NCD_Operator_Compressed_PMO_II_10}
\liminf_{n\in\mathbb{N}}\hspace{0.025cm} m_{f,\varepsilon_{n}}^{-\theta}\lc{}t,0\rc{}=\infty
\end{align}

\noindent for all $t\in\spec_{L^{\infty}(A,\tau)} x$. Applying Fatou's Lemma to Equation \ref{EQ.LEM.NCD_Operator_Compressed_PMO_II_9}, Equation \ref{EQ.LEM.NCD_Operator_Compressed_PMO_II_10} shows

\begin{align}\label{EQ.LEM.NCD_Operator_Compressed_PMO_II_11}
\liminf_{n\in\mathbb{N}}\hspace{0.025cm} \lgl\mathcal{D}_{x,0,\varepsilon_{n}}^{\theta}(u),u\rgl_{\omega}\geq\int_{\spec_{L^{\infty}(A,\tau)}x\times 0}\liminf_{n\in\mathbb{N}}\hspace{0.025cm} m_{f,\varepsilon_{n}}^{-\theta}\lc{}t,0\rc{}dE_{L_{x}^{\phi},0}^{u}=\infty.
\end{align}

\noindent The sequence used for Equation \ref{EQ.LEM.NCD_Operator_Compressed_PMO_II_10} is fixed but arbitrary. As such, Equation \ref{EQ.LEM.NCD_Operator_Compressed_PMO_II_7} and Equation \ref{EQ.LEM.NCD_Operator_Compressed_PMO_II_11} imply $2)$ since no descending sequence yields a finite value.
\end{proof}

\begin{thm}\label{THM.NCD_Operator_Compressed_PMO}
Let $(A,\tau)$ and $(B,\omega)$ be tracial AF-$C^{*}$-algebras. Let $(\phi,\bpsi,\gamma)$ be an AF-$A$-bimodule structure on $B$. Let $f$ be representing function of an operator mean and $\theta\in [0,1]$. Let $p\in L^{\infty}(A,\tau)$ be a projection. For all $x,y\in L^{1}(A[p],\tau)_{+}$, we have

\begin{itemize}
\item[1)] $H\big(Q_{x^{\flat},y^{\flat}}^{f,\theta}\big)\subset L^{2}(B[p],\omega)$,

\item[2)] $\mathcal{D}_{x^{\flat},y^{\flat}}^{\theta}=\sr$-$\lim_{\varepsilon\downarrow 0}\mathcal{D}_{x,y,p,\varepsilon}^{\theta}$ on $H\big(Q_{x^{\flat},y^{\flat}}^{f,\theta}\big)$,

\item[3)] $Q_{x^{\flat},y^{\flat}}^{f,\theta}(u)=\sup_{\varepsilon>0}\lgl\mathcal{D}_{x,y,p,\varepsilon}^{\theta}(u),u\rgl_{\omega}$ for all $u\in H\big(Q_{x^{\flat},y^{\flat}}^{f,\theta}\big)$.
\end{itemize}
\end{thm}
\begin{proof}
Let $x,y\in L^{1}(A,\tau)_{+}$. We show $1)$. Theorem \ref{THM.QE_QForm_Representation} and Corollary \ref{COR.AF_Cstar_Bimodule_CLRA_SR} imply

\begin{align}\label{EQ.THM.NCD_Operator_Compressed_PMO_1}
Q_{x^{\flat},y^{\flat}}^{f,\theta}(u)=\sup_{\varepsilon>0}\hspace{0.025cm} \lgl\mathcal{D}_{x,y,\varepsilon}^{f,\theta}(u),u\rgl_{\omega}
\end{align}

\noindent for all $u\in L^{2}(B,\omega)$.\par


\pagebreak


Let $u\in L^{2}(B,\omega)$ and $\varepsilon>0$. Equation \ref{EQ.LEM.AF_NCD_FC_II_2} shows

\begin{align}\label{EQ.THM.NCD_Operator_Compressed_PMO_2}
L^{2}(B[p],\omega)^{\perp}=pL^{2}(B,\omega)p^{\perp}\oplus p^{\perp}L^{2}(B,\omega)p\oplus L^{2}\lc{}B\lc{}p^{\perp}\rc{},\omega\rc{}.
\end{align}

\noindent Using $4)$ in Lemma \ref{LEM.NCD_Operator_Compressed_PMO_I} and $1)$ in Lemma \ref{LEM.NCD_Operator_Compressed_PMO_II}, we have

\begin{align}\label{EQ.THM.NCD_Operator_Compressed_PMO_3}
\mathcal{D}_{x,y,\varepsilon}^{\theta}=\mathcal{D}_{x,y,p,\varepsilon}^{\theta}\oplus\lc\mathcal{D}_{x,0,\varepsilon}^{\theta}L_{p}^{\phi}\mathrlap{\phantom{R}^{\bpsi}}R_{p^{\perp}}\oplus\mathcal{D}_{0,y,\varepsilon}^{\theta}\mathrlap{\phantom{L}^{\phi}}L_{p^{\perp}}R_{p}^{\bpsi}\oplus\varepsilon^{-\theta}\pi_{p^{\perp}}\rc{}
\end{align}

\noindent w.r.t.~$\BII\lc{}L^{2}(B[p],\omega)\rc\oplus\BII\lc{}pL^{2}(B,\omega)p^{\perp}\rc\oplus\BII\lc{}p^{\perp}L^{2}(B,\omega)p\rc\oplus\BII\lc{}L^{2}\lc{}B\lc{}p^{\perp}\rc{},\omega\rc\rc$. Moreover, all bounded operators in Equation \ref{EQ.THM.NCD_Operator_Compressed_PMO_3} are positive. Equation \ref{EQ.THM.NCD_Operator_Compressed_PMO_3} lets us estimate

\begin{align*}
\lgl\mathcal{D}_{x,y,\varepsilon}^{\theta}(u),u\rgl_{\omega} & = \dblv{}\lc\mathcal{D}_{x,0,\varepsilon}^{\theta}L_{p}^{\phi}\mathrlap{\phantom{R}^{\bpsi}}R_{p^{\perp}}\oplus\mathcal{D}_{0,y,\varepsilon}^{\theta}\mathrlap{\phantom{L}^{\phi}}L_{p^{\perp}}R_{p}^{\bpsi}\oplus\varepsilon^{-\theta}\pi_{p^{\perp}}\rc^{\frac{1}{2}}\lc\pi_{p}^{\perp}(u)\rc\dblv_{\omega}^{2}+\dblv{}\mathcal{D}_{x,y,\varepsilon}^{\frac{\theta}{2}}\pi_{p}(u)\dblv_{\omega}^{2} \phantom{\bigg)} \\
& \geq \dblv{}\lc\mathcal{D}_{x,0,\varepsilon}^{\theta}L_{p}^{\phi}\mathrlap{\phantom{R}^{\bpsi}}R_{p^{\perp}}\oplus\mathcal{D}_{0,y,\varepsilon}^{\theta}\mathrlap{\phantom{L}^{\phi}}L_{p^{\perp}}R_{p}^{\bpsi}\oplus\varepsilon^{-\theta}\pi_{p^{\perp}}\rc^{\frac{1}{2}}\lc\pi_{p}^{\perp}(u)\rc\dblv_{\omega}^{2} \phantom{\bigg)} \\
& = \dblv{}\mathcal{D}_{x,0,\varepsilon}^{\frac{\theta}{2}}\lc{}pup^{\perp}\rc\dblv_{\omega}^{2}+\dblv{}\mathcal{D}_{0,y,\varepsilon}^{\frac{\theta}{2}}\lc{}p^{\perp}up\rc\dblv_{\omega}^{2}+\varepsilon^{-\theta}\dblv{}\pi_{p^{\perp}}(u)\dblv_{\omega}^{2}. \phantom{\bigg)}
\end{align*}

\noindent Using Equation \ref{EQ.THM.NCD_Operator_Compressed_PMO_1} and $2)$ in Proposition \ref{PRP.NCD_Operator_Compressed_PMO_II}, taking suprema in $\varepsilon>0$ yields

\begin{align}\label{EQ.THM.NCD_Operator_Compressed_PMO_4}
Q_{x^{\flat},y^{\flat}}^{f,\theta}(u)\geq\sup_{\varepsilon>0}\hspace{0.025cm} \dblv{}\mathcal{D}_{x,0,\varepsilon}^{\frac{\theta}{2}}\lc{}pup^{\perp}\rc\dblv_{\omega}^{2}+\sup_{\varepsilon>0}\hspace{0.025cm} \dblv{}\mathcal{D}_{0,y,\varepsilon}^{\frac{\theta}{2}}\lc{}p^{\perp}up\rc\dblv_{\omega}^{2}+\sup_{\varepsilon>0}\hspace{0.025cm} \varepsilon^{-\theta}\dblv{}\pi_{p^{\perp}}(u)\dblv_{\omega}^{2}.
\end{align}

\noindent Using $2)$ in Lemma \ref{LEM.NCD_Operator_Compressed_PMO_II}, Equation \ref{EQ.THM.NCD_Operator_Compressed_PMO_4} implies $1)$ at once. We therefore get $2)$ and $3)$ by Equation \ref{EQ.THM.NCD_Operator_Compressed_PMO_1}, Equation \ref{EQ.THM.NCD_Operator_Compressed_PMO_3} and Theorem \ref{THM.QE_QForm_Representation}.
\end{proof}

\begin{cor}\label{COR.NCD_Operator_Compressed_PMO_I}
Let $p\in L^{\infty}(A,\tau)$ be a projection. If $x,y\in L^{0}(A[p],\tau)_{+}$ s.t.~we have $m_{f}^{-1}\in\SIIp\lc{}E_{x,y}\rc$, then

\begin{itemize}
\item[1)] $H\big(Q_{x^{\flat},y^{\flat}}^{f,\theta}\big)=L^{2}(B[p],\omega)$,

\item[2)] $\mathcal{D}_{x^{\flat},y^{\flat}}^{\theta}=\mathcal{D}_{x,y,p}^{\theta}=\sr$-$\lim_{\varepsilon\downarrow 0}\mathcal{D}_{x,y,p,\varepsilon}^{\theta}$ on $L^{2}(B[p],\omega)$,

\item[3)] $Q_{x^{\flat},y^{\flat}}^{f,\theta}(u)=\big\|\mathcal{D}_{x,y,p}^{\frac{\theta}{2}}(u)\big\|_{\omega}^{2}=\sup_{\varepsilon>0}\lgl\mathcal{D}_{x,y,p,\varepsilon}^{\theta}(u),u\rgl_{\omega}$ for all $u\in L^{2}(B[p],\omega)$.
\end{itemize}
\end{cor}
\begin{proof}
Note $\mathcal{D}_{x,y,p}^{\theta}=\sr$-$\lim_{\varepsilon\downarrow 0}\mathcal{D}_{x,y,p,\varepsilon}^{\theta}$ on $L^{2}(B[p],\omega)$ by $3)$ in Proposition \ref{PRP.NCD_Operator_Compressed_PMO_II}. Thus $2)$ in Proposition \ref{PRP.NCD_Operator_Compressed_PMO_II} and the Kato-Robinson theorem imply

\begin{align}\label{EQ.COR.NCD_Operator_Compressed_PMO_I_1}
\sup_{\varepsilon>0}\hspace{0.025cm} \lgl\mathcal{D}_{x,y,p,\varepsilon}^{\theta}(u),u\rgl_{\omega}=\dblv{}\mathcal{D}_{x,y,p}^{\frac{\theta}{2}}(u)\dblv_{\omega}^{2}<\infty    
\end{align}

\noindent for all $u\in\dom\mathcal{D}_{x,y,p}^{\frac{\theta}{2}}$. We know $\mathcal{D}_{x,y,p}^{\theta}$ is densely defined as per $1)$ in Proposition \ref{PRP.NCD_Operator_Compressed_PMO_I} by hypothesis and construction of compressed pulled-back joint functional calculus. Using Equation \ref{EQ.COR.NCD_Operator_Compressed_PMO_I_1}, Theorem \ref{THM.NCD_Operator_Compressed_PMO} hence implies our claims.
\end{proof}

\begin{rem}\label{REM.NCD_Operator_Compressed_PMO_I}
If $x,y\in L^{1}(A[p],\tau)_{+}$ s.t.~$L_{x,p}$ and $R_{y,p}$ injective, then $m_{f}^{-1}\in\SIIp\lc{}E_{x,y}\rc$. Since $E_{x,L^{\infty}(A[p],\tau)}$ and $E_{y,L^{\infty}(A[p],\tau)}$ have no mass at zero if injectivity is given, we know $[0,\infty)\times\lset{}0\rset\cup\lset{}0\rset\times [0,\infty)\in\mathcal{N}\lc{}E_{x,y,L^{\infty}(A[p],\tau)}\rc$. Zero may still lie in $\spec_{L^{\infty}(A[p],\tau)}x\times y$.
\end{rem}

\begin{cor}\label{COR.NCD_Operator_Compressed_PMO_II}
Let $p,q\in L^{\infty}(A,\tau)$ be projections. If $x,y\in L^{0}(A[p],\tau)_{+}\cap L^{0}\lc{}A\lc{}q\rc{},\tau\rc_{+}$ s.t.~$m_{f}^{-1}\in\SIIp\lc{}E_{x,y}\rc\cap\SIIq\lc{}E_{x,y}\rc$, then $p=q$.
\end{cor}
\begin{proof}
We equip $A$ with its canonical AF-$A$-bimodule structure. We have $\pi_{p}=\pi_{q}$ by $1)$ in Corollary \ref{COR.NCD_Operator_Compressed_PMO_I}. For all $j\in\mathbb{N}$, note $p1_{A_{j}}p=q1_{A_{j}}q$. Using sequential strong continuity of multiplication and Proposition \ref{PRP.AF_Cstar_Unit}, get $p=\s$-$\lim_{j\in\mathbb{N}}p1_{A_{j}}p=\s$-$\lim_{j\in\mathbb{N}}q1_{A_{j}}q=q$.
\end{proof}


\section{Noncommutative gradients}\label{SEC.NCDS_NCG}

Symmetric $C^{*}$-derivations are noncommutative gradients \cite{ART.Cip.1997.NC_Dirichlet_Markov}\cite{ART.Cip_Sav.2003.NC_Dirichlet_Grad}. We introduce and consider the special case of symmetric $W^{*}$-derivations. Using symmetric $W^{*}$-bimodules induced by AF-$C^{*}$-bimodules as codomains, quantum gradients are, by construction, a class of symmetric $W^{*}$-derivations compatible with compression and finite-dimensional approximation. Their dualisation provides the weak formulation of continuity equations in the AF-$C^{*}$-setting. Thus Banach dual spaces of AF-$C^{*}$-bimodules serve as synthetic tangent spaces. Compatibility transfers to quantum Laplacians, their noncommutative heat semigroups, as well as continuity equations. Compatibility therefore transfers to quantum optimal transport.

\medskip

\noindent\textbf{Structure.} In Subsection \ref{SSEC.NCDS_NCG_Ubd_Derivation}, we review symmetric $C^{*}$-~and $W^{*}$-derivations. We study their compression. In Subsection \ref{SSEC.NCDS_NCG_QG}, we define quantum gradients, collect properties and give standard constructions. We further construct dynamic quantum gradients from twisted conjugation groups. In Subsection \ref{SSEC.NCDS_NCG_Notion}, we define noncommutative differential structures, discuss compatibility and outline the coarse graining process.


\subsection[Symmetric $C^{*}$-~and $W^{*}$-derivations]{Symmetric $\mathbf{C}^{*}$-~and $\mathbf{W}^{*}$-derivations}\label{SSEC.NCDS_NCG_Ubd_Derivation}

Symmetric $C^{*}$-derivations are closable unbounded module derivations for symmetric $C^{*}$-bimodules intertwining adjoining and anti-linear involution. They determine noncommutative analogues of Dirichlet forms \cite{BK.Fuk_Osh_Tak.2011.Dirichlet_Markov}, called $C^{*}$-Dirichlet forms \cite{ART.Alb_Ser.1977.Cstar_Dirichlet_Markov}\cite{ART.Cip.1997.NC_Dirichlet_Markov}\cite{ART.Cip_Sav.2003.NC_Dirichlet_Grad}. Following likewise generalised Beurling-Deny formula \cite{ART.Beu_Den.1959.Dirichlet}, representing operators of conservative $C^{*}$-Dirichlet forms are concatenations of symmetric $C^{*}$-derivations and their adjoints \lc{}cf.~Theorem 8.3 in \cite{ART.Cip_Sav.2003.NC_Dirichlet_Grad}\rc{}. These in turn generate completely Markovian semigroups for tracial $C^{*}$-algebras \lc{}cf.~Theorem 4.11 in \cite{ART.Cip.1997.NC_Dirichlet_Markov}\rc{}. Altogether, we say that symmetric $C^{*}$-derivations are noncommutative gradients which determine Laplacians and view completely Markovian semigroups generated by the latter as noncommutative heat semigroups. The relationship between gradients, heat semigroups and Dirichlet forms extends to the noncommutative setting \cite{ART.Cip.1997.NC_Dirichlet_Markov}\cite{ART.Cip_Sav.2003.NC_Dirichlet_Grad}.\par


\pagebreak


We define symmetric $W^{*}$-derivations to be symmetric $C^{*}$-derivations for symmetric $W^{*}$-bimodules, moreover closable w.r.t.~bounded strong convergence s.t.~units are in the kernel upon closure. Using symmetric $W^{*}$-bimodules induced by AF-$C^{*}$-bimodules as codomains, we have compression based on compression of AF-$C^{*}$-bimodules. We then define quantum gradients to be symmetric $W^{*}$-derivations with sufficient compression to have finite-dimensional approximation. Standard references for $C^{*}$-bimodules and $C^{*}$-derivations are \cite{ART.Cip.1997.NC_Dirichlet_Markov}\cite{ART.Cip_Sav.2003.NC_Dirichlet_Grad}. The latter are collected in \cite{COL.Cip.2008.NC_Dirichlet} on p.161-276 in \cite{BK.Cip.2008.NC_Dirichlet}.


\subsubsection*{Unbounded module derivations}

Definition \ref{DFN.CWstar_Derivation} collects notions of unbounded module derivations we use, including symmetric $C^{*}$-~and $W^{*}$-derivations. This yields a more general definition of symmetric $C^{*}$-derivations than in \cite{ART.Cip_Sav.2003.NC_Dirichlet_Grad}. Remark \ref{REM.CWstar_Derivation} shows results for symmetric $C^{*}$-derivations in \cite{ART.Cip.1997.NC_Dirichlet_Markov}\cite{ART.Cip_Sav.2003.NC_Dirichlet_Grad} apply regardless. Proposition \ref{PRP.Wstar_Derivation_Chain} states the chain rule for symmetric $W^{*}$-derivations.\par
Let $(M,\tau)$ be a tracial $W^{*}$-algebra.

\begin{ntn}\label{NTN.Ubd_Closure}
Unless stated otherwise, we use the identical symbols for unbounded operators and all of their closures. For all closable unbounded operators $T:H_{0}\longrightarrow H_{1}$ of Hilbert spaces, let $\|.\|_{T}$ denote its graph norm.
\end{ntn}

\begin{dfn}\label{DFN.CWstar_Derivation}
Let $A\subset M$ be a $\sigma$-weakly dense $C^{*}$-subalgebra and $H$ a symmetric $W^{*}$-bimodule over $M$. Let $\mathcal{A}\subset A$ be a $^{*}$-subalgebra and $\nabla:\mathcal{A}\longrightarrow H$ a linear map.

\begin{itemize}
\item[1)] We say that $\nabla$ satisfies

\begin{itemize}
\item[1.1)] the Leibniz rule if $\nabla xy=\nabla x\cdot y+x\nabla y$ for all $x,y\in\mathcal{A}$,

\item[1.2)] symmetry if $\nabla x^{*}=\gamma\lc\nabla x\rc$ for all $x\in\mathcal{A}$.
\end{itemize}

\begin{reapply}
\end{reapply}

\item[2)] We say that $\nabla$ is an $\mathcal{A}$-module derivation if it satisfies the Leibniz rule, and further call $\nabla$ symmetric if it satisfies symmetry.

\item[3)] We say that $\nabla$ is a symmetric $C^{*}$-derivation if

\begin{itemize}
\item[3.1)] $\nabla$ is a symmetric $\mathcal{A}$-module derivation,

\item[3.2)] $\mathcal{A}\subset A$ is $\|.\|_{A}$-dense and $\mathcal{A}\subset L^{2}(M,\tau)$ is $\|.\|_{\tau}$-dense,

\item[3.3)] $\nabla$ is $\lc\|.\|_{A},\|.\|_{H}\rc$-closable and $\nabla\vert_{\mathcal{A}\cap L^{2}(M,\tau)}$ is $\lc\|.\|_{\tau},\|.\|_{H}\rc$-closable.
\end{itemize}

\begin{reapply}
\end{reapply}

\noindent Then its $\lc\|.\|_{\tau},\|.\|_{H}\rc$-closure defines the Laplacian $\Delta:=\nabla^{*}\nabla$ of $\nabla$.

\item[4)] We say that $\nabla$ is a symmetric $W^{*}$-derivation if

\begin{itemize}
\item[4.1)] $\nabla$ is a symmetric $C^{*}$-derivation,

\item[4.2)] for all nets $\{x_{k}\}_{k\in K}\subset\mathcal{A}$ s.t.~$\bds$-$\lim_{k\in K}x_{k}=\bds$-$\lim_{k\in K}x_{k}^{*}=0$, get existence of $\|.\|_{H}$-$\lim_{k\in K}\nabla x_{k}$ if and only if $\|.\|_{H}$-$\lim_{k\in K}\nabla x_{k}=0$,

\item[4.3)] there exists a net $\{x_{k}\}_{k\in K}\subset\mathcal{A}$ s.t.~$\bds$-$\lim_{k\in K}x_{k}=\bds$-$\lim_{k\in K}x_{k}^{*}=1_{M}$ and $\|.\|_{H}$-$\lim_{k\in K}\nabla x_{k}=0$.
\end{itemize}

\begin{reapply}
\end{reapply}

\end{itemize}
\end{dfn}

\begin{rem}\label{REM.CWstar_Derivation}
Assume the setting of Definition \ref{DFN.CWstar_Derivation}. Let $\nabla:\mathcal{A}\longrightarrow H$ be a symmetric $C^{*}$-derivation. Restricting the bimodule action of $M$ to $A$ yields symmetric $C^{*}$-bimodule $H$ over $A$ as per Definition \ref{DFN.CWstar_Bimodule}. If $\tau\vert_{A_{+}}$ is semi-finite, then $(A,\tau)$ is tracial $C^{*}$-algebra and $\nabla$ is symmetric $C^{*}$-derivation used in \cite{ART.Cip_Sav.2003.NC_Dirichlet_Grad}. Since semi-finiteness does not affect the chain rule, the relationship between gradients, heat semigroups and Dirichlet forms in the noncommutative setting uses the $\lc\|.\|_{\tau},\|.\|_{H}\rc$-closure of $\nabla$. We therefore know results for symmetric $C^{*}$-derivations in \cite{ART.Cip.1997.NC_Dirichlet_Markov}\cite{ART.Cip_Sav.2003.NC_Dirichlet_Grad} apply to our general notion. However, we apply them only if $A$ is unital and $\tau<\infty$. Note semi-finiteness of $\tau\vert_{A_{+}}$ is always given in this case \lc{}cf.~$2)$ in Proposition \ref{PRP.Wstar_Trace_Fin_I}\rc{}. If $\tau<\infty$, then replacing $\lc\|.\|_{A},\|.\|_{H}\rc$-closable in $3.3)$ in Definition \ref{DFN.CWstar_Derivation} by $\lc\|.\|_{A},\|.\|_{H}\rc$-closed yields identical $\lc\|.\|_{\tau},\|.\|_{H}\rc$-closures.
\end{rem}

\begin{dfn}\label{DFN.Wstar_Derivation_BdCon_I}
Let $A\subset M$ be a $\sigma$-weakly dense $C^{*}$-subalgebra and $H$ a symmetric $W^{*}$-bimodule over $M$. Let $\nabla:\mathcal{A}\longrightarrow H$ be a symmetric $W^{*}$-derivation.

\begin{itemize}
\item[1)] We call a net $\{x_{k}\}_{k\in K}\subset\mathcal{A}$ bounded strongly convergent to $x\in M$ for $\nabla$ if

\begin{itemize}
\item[1.1)] $x=\bds$-$\lim_{k\in K}x_{k}$ and $x^{*}=\bds$-$\lim_{k\in K}x_{k}^{*}$,

\item[1.2)] $\lset\nabla x_{k}\rset_{k\in K}\subset H$ is Cauchy net in norm.
\end{itemize}

\begin{reapply}
\end{reapply}

Let $x=\bds^{\nabla}$-$\lim_{k\in K}x_{k}$ denote bounded strong convergence for $\nabla$.

\item[2)] Set $M_{\nabla}:=\big\{\hspace{0.025cm} x\in M\ \vset\ \exists\{x_{k}\}_{k\in K}\subset\mathcal{A}:\ x=\bds^{\nabla}\textrm{-}\lim_{k\in K}x_{k}\hspace{0.025cm} \big\}$.
\end{itemize}
\end{dfn}

Let $A\subset M$ be a $\sigma$-weakly dense $C^{*}$-subalgebra and $H$ a symmetric $W^{*}$-bimodule over $M$. Let $\nabla:\mathcal{A}\longrightarrow H$ be a symmetric $W^{*}$-derivation. By $4.2)$ in Definition \ref{DFN.CWstar_Derivation}, we define bounded strong closure $\nabla:M_{\nabla}\longrightarrow H$ of $\nabla$ by setting

\begin{align}\label{EQ.SSEC.NCDS_NCG_Ubd_Derivation_1}
\nabla x:=\|.\|_{H}\textrm{-}\lim_{k\in K}\hspace{0.025cm} \nabla x_{k} 
\end{align}

\noindent for all $x\in M_{\nabla}$. In each case, we use fixed but arbitrary net $\{x_{k}\}_{k\in K}\subset\mathcal{A}$ bounded strongly convergent to $x\in M$ for $\nabla$.

\begin{dfn}\label{DFN.Wstar_Derivation_BdCon_II}
Let $A\subset M$ be a $\sigma$-weakly dense $C^{*}$-subalgebra and $H$ a symmetric $W^{*}$-bimodule over $M$. For all symmetric $W^{*}$-derivations $\nabla:\mathcal{A}\longrightarrow H$, its bounded strong closure $\nabla:M_{\nabla}\longrightarrow H$ is defined by Equation \ref{EQ.SSEC.NCDS_NCG_Ubd_Derivation_1}.
\end{dfn}

\begin{prp}\label{PRP.Wstar_Derivation_BdCon}
Let $A\subset M$ be a $\sigma$-weakly dense $C^{*}$-subalgebra and $H$ a symmetric $W^{*}$-bimodule over $M$. For all symmetric $W^{*}$-derivations $\nabla:\mathcal{A}\longrightarrow H$, we have

\begin{itemize}
\item[1)] $1_{M}\in M_{\nabla}$ and unital $^{*}$-subalgebra $M_{\nabla}\subset M$,

\item[2)] symmetric $M_{\nabla}$-module derivation $\nabla:M_{\nabla}\longrightarrow H$ and $\nabla 1_{M}=0$.
\end{itemize}
\end{prp}
\begin{proof}
Multiplication in $M$ is jointly continuous in strong operator topology. We thus know $M_{\nabla}\subset M$ is a $^{*}$-subalgebra, further having unit $1_{M}\in M_{\nabla}$ with $\nabla 1_{M}=0$ by $4.3)$ in Definition \ref{DFN.CWstar_Derivation}. Since we use normal unital $^{*}$-homomorphisms to define the bimodule action of $M$ on $H$ as per Definition \ref{DFN.CWstar_Bimodule}, the Leibniz rule extends from $\mathcal{A}$ to $M_{\nabla}$. Note symmetry follows by construction. Altogether, get $1)$ and $2)$.
\end{proof}


\pagebreak


The Leibniz rule formulates a noncommutative chain rule using functional calculus of left-~and right-bimodule actions of symmetric $W^{*}$-bimodules. Following notation in Definition \ref{DFN.CWstar_Bimodule}, we use $(\phi,\bpsi)$-action of $M$ on $H$ for normal unital $^{*}$-homomorphisms $\phi,\bpsi:M\longrightarrow\BII(H)$. For all $x,y\in M$, $\phi(x),\bpsi(y)\in\BII(H)$ commute by definition.

\begin{dfn}\label{DFN.QT_Derivation}
Let $I\subset\mathbb{R}$ be a closed interval. For all $g\in C^{1}(I)$, we define functional derivative of $g$ on $I\times I$ by setting

\begin{align*}
Dg(t,s):=
\begin{cases}
\frac{g(t)-g(s)}{t-s} & \If\ t\neq s, \\
\frac{d}{dt}g(t) & \Else. \phantom{\bigg)}
\end{cases}
\end{align*}
\end{dfn}

\begin{rem}\label{REM.QT_Derivation}
Note $Dg\in C(I\times I)$ s.t.~$\big\| Dg\big\|_{C(I\times I)}\leq \big\|\frac{d}{dt}g\big\|_{C(I)}$ in each case.
\end{rem}

\begin{prp}\label{PRP.QT_Derivation}
Let $x\in M_{h}$. If $I\subset\mathbb{R}$ is a closed interval s.t.~$\specM x\subset I$, then

\begin{itemize}
\item[1)] $C(I\times I)\subset C\lc\specM x\times\specM x\rc\subset L^{\infty}\lc\spec\phi(x)\times\bpsi(x),dE_{\phi(x),\bpsi(x)}\rc$,

\item[2)] $\big\|\Gamma_{\phi(x),\bpsi(x)}(h)\big\|_{\BII(H)}\leq \|h\|_{C(I\times I)}$ for all $h\in C(I\times I)$.
\end{itemize}
\end{prp}
\begin{proof}
Note $\spec\phi(x),\spec\bpsi(x)\subset\specM x$ as $\phi$ and $\bpsi$ are unital $^{*}$-homomorphisms. Get 

\begin{align}\label{EQ.PRP.QT_Derivation_1}
\spec\phi(x)\times\bpsi(x)\subset\specM x\times\specM x\subset I\times I.   
\end{align}

\noindent Equation \ref{EQ.PRP.QT_Derivation_1} implies $1)$ by dualisation. Bounded measurable joint functional calculus $\Gamma_{\phi(x),\bpsi(x)}:L^{\infty}\lc\spec\phi(x)\times\bpsi(x),dE_{\phi(x),\bpsi(x)}\rc\longrightarrow\BII(H)$ is a normal unital $^{*}$-homomorphism \lc{}cf.~$1)$ in Proposition \ref{PRP.JFC_Bd}\rc{}. Using $\dblv{}\Gamma_{\phi(x),\bpsi(x)}\dblv{}\leq 1$ and $1)$, we obtain $2)$ at once.
\end{proof}

\begin{prp}\label{PRP.Wstar_Derivation_Chain}
Let $A\subset M$ be a $\sigma$-weakly dense $C^{*}$-subalgebra and $H$ a symmetric $W^{*}$-bimodule over $M$. Let $\nabla:\mathcal{A}\longrightarrow H$ be a symmetric $W^{*}$-derivation, $x\in M_{\nabla}$ self-adjoint and $I\subset\mathbb{R}$ a closed interval s.t.~$\specM x\subset I$. If $g\in C^{1}(I)$, then

\begin{itemize}
\item[1)] $g(x)\in M_{\nabla}$ self-adjoint and $\nabla g(x)=\Gamma_{\phi(x),\bpsi(x)}\lc{}Dg\rc\lc\nabla x\rc$,

\item[2)] $\dblv{}\nabla g(x)\dblv_{H}\leq \big\|\frac{d}{dt}g\big\|_{C(I)}\cdot \|\nabla x\|_{H}$.
\end{itemize}
\end{prp}
\begin{proof}
Note $\nabla 1_{M}=0$ by $2)$ in Proposition \ref{PRP.Wstar_Derivation_BdCon}. If $g$ is polynomial, then we directly verify $1)$ and $2)$ using the Leibniz rule, symmetry, and $\nabla 1_{M}=0$. Let $I=[a,b]$ for $a\leq b$ in $\mathbb{R}$ and $g\in C^{1}(I)$. Since $\nabla 1_{M}=0$, we assume $g\lc{}a\rc{}=0$ without loss of generality.\par
We know $\frac{d}{dt}g\in C(I)$. Let $\{q_{n}\}_{n\in\mathbb{N}}\subset C(\mathbb{R})$ be polynomials s.t.~$\frac{d}{dt}g=\|.\|_{\infty}$-$\lim_{n\in\mathbb{N}}q_{n}$. For all $n\in\mathbb{N}$, set $g_{n}(t):=\int_{a}^{t}q_{n}(s)ds$ for all $t\in I$. Get $g_{n}\in C^{1}(I)$ with derivative $q_{n}$ in each case. Using standard arguments for integration \cite{BK.Eva.2010.Partial_Differential_Equations}\cite{BK.Koe.1993.Analysis_II}\cite{BK.Koe.2004.Analysis_I}, norm convergence of derivatives implies $g=\|.\|_{\infty}$-$\lim_{n\in\mathbb{N}}g_{n}$ since $g\lc{}a\rc{}=g_{n}\lc{}a\rc{}=0$ for all $n\in\mathbb{N}$. Following our definition of bounded strong closure as per Equation \ref{EQ.SSEC.NCDS_NCG_Ubd_Derivation_1}, such approximation reduces our claims to the polynomial case by linearity of the functional derivative and Proposition \ref{PRP.QT_Derivation}.
\end{proof}


\subsubsection*{Compressing symmetric $\mathbf{W}^{*}$-derivations}

Definition \ref{DFN.Wstar_Derivation_Compression_I} gives compression of symmetric $W^{*}$-derivations. It is based on compression of AF-$C^{*}$-bimodules. The two classes of compression given in Subsection \ref{SSEC.NCDS_AF_FC} each provide compression of symmetric $W^{*}$-derivations. First, we compress to induced AF-$C^{*}$-bimodules in Corollary \ref{COR.Wstar_Derivation_Compression_Restriction}. Secondly, we compress with projections in Corollary \ref{COR.Wstar_Derivation_Projection}.\par
Let $(A,\tau)$ and $(B,\omega)$ be tracial AF-$C^{*}$-algebras. Let $(\phi,\bpsi,\gamma)$ be an AF-$A$-bimodule structure on $B$. Let $\nabla:\mathcal{A}\longrightarrow L^{2}(B,\omega)$ be a symmetric $W^{*}$-derivation.

\begin{lem}\label{LEM.Reducible_Operator}
Let $H_{0}$ and $H_{1}$ be Hilbert spaces, $V_{0}\subset H_{0}$ and $V_{1}\subset H_{1}$ Hilbert subspaces, and $T:H_{0}\longrightarrow H_{1}$ closed unbounded operator. If $\CII$ is core of $T$ s.t.~

\begin{itemize}
\item[1)] $\pi_{V_{0}}^{H_{0}}\lc\CII\rc\subset\dom T$,

\item[2)] $\pi_{V_{1}}^{H_{1}}\big(T(x)\big)=T\big(\pi_{V_{0}}^{H_{0}}(x)\big)$ for all $x\in\CII$,
\end{itemize}

\noindent then $\pi_{V_{1}}^{H_{1}}T\subset T\pi_{V_{0}}^{H_{0}}$.
\end{lem}
\begin{proof}
Let $x\in\dom T$ and $x=\|.\|_{T}$-$\lim_{k\in K}x_{k}$ for a net $\{x_{k}\}_{k\in K}\subset\CII$. Using $1)$ and $2)$, get

\begin{align}\label{EQ.LEM.Reducible_Operator_1}
\pi_{V_{1}}^{H_{1}}\big(T(x)\big)=\|.\|_{H_{1}}\textrm{-}\lim_{k\in K}\hspace{0.025cm} \pi_{V_{1}}^{H_{1}}\big(T(x_{k})\big)=\|.\|_{H_{1}}\textrm{-}\lim_{k\in K}\hspace{0.025cm} T\big(\pi_{V_{0}}^{H_{0}}(x_{k})\big).
\end{align}

\noindent Equation \ref{EQ.LEM.Reducible_Operator_1} shows $\pi_{V_{1}}^{H_{1}}\big(T(x)\big)=T\big(\pi_{V_{0}}^{H_{0}}(x)\big)$ since $T$ is closed.
\end{proof}

\begin{rem}\label{REM.Reducible}
Assume $H:=H_{0}=H_{1}$ and $V:=V_{0}=V_{1}$ in the setting of Lemma \ref{LEM.Reducible_Operator}. If $T\in\UBII(H)_{h}$ has core as per Lemma \ref{LEM.Reducible_Operator}, then $T$ is $V$-reducible. If $T\in\UBII_{V}(H)$ and $\CII$ core of $T$, then $\CII$ satisfies $1)$ and $2)$ in Lemma \ref{LEM.Reducible_Operator}.
\end{rem}

\begin{dfn}\label{DFN.Wstar_Derivation_Compression_I}
Let $(\phi,\bpsi,\gamma)$ be $(N,V)$-compressible. We say that $\nabla:\mathcal{A}\longrightarrow L^{2}(B,\omega)$ is $(N,V)$-compressible, and call $(N,V)$ a compression of $\nabla$, if

\begin{itemize}
\item[1)] $\pi_{L^{2}(N,\tau)}^{A}\lc\mathcal{A}\rc\subset N\cap L^{\infty}(A,\tau)_{\nabla}$,

\item[2)] $\pi_{L^{2}(N,\tau)}^{A}\lc\mathcal{A}\rc\subset N$ is $\sigma$-weakly dense and $\pi_{L^{2}(N,\tau)}^{A}\lc\mathcal{A}\rc\subset L^{2}(N,\tau)$ is $\|.\|_{\tau}$-dense,

\item[3)] $\pi_{L^{2}(N,\tau)}^{A}\lc\mathcal{A}\rc\subset\dom\nabla$, $\pi_{V}^{B}\lc\dom\nabla^{*}\rc\subset\dom\nabla^{*}$, and

\begin{align}\label{EQ.DFN.Wstar_Derivation_Compression_I_1}
\pi_{V}^{B}\lc\nabla x\rc{}=\nabla\pi_{L^{2}(N,\tau)}^{A}(x),\ \pi_{L^{2}(N,\tau)}^{A}\big(\nabla^{*}u\big)=\nabla^{*}\pi_{V}^{B}(u)
\end{align}

\begin{reapply}
\end{reapply}

\noindent for all $x\in\mathcal{A}$ and $u\in\dom\nabla^{*}$.
\end{itemize}
\end{dfn}


\pagebreak


\begin{cor}\label{COR.Wstar_Derivation_Compression_Restriction}
Let $\mathcal{A}=A_{0}$. If $j\in\mathbb{N}$ s.t.~

\begin{align}\label{EQ.COR.Wstar_Derivation_Compression_Res_1}
\nabla(A_{j})\subset B_{j},\ \nabla^{*}(B_{j})\subset A_{j},
\end{align}

\noindent then $\nabla$ is $(A_{j},B_{j})$-compressible.
\end{cor}
\begin{proof}
Let $j\in\mathbb{N}$ s.t.~Equation \ref{EQ.COR.Wstar_Derivation_Compression_Res_1} holds. We know $(\phi,\bpsi,\gamma)$ is $(A_{j},B_{j})$-compressible by Corollary \ref{COR.AF_Cstar_Bimodule_Compression_Restriction}. Using $\pi_{L^{2}(N,\tau)}^{A}=\pi_{j}^{A}$, $\pi_{V}^{B}=\pi_{j}^{B}$ and Equation \ref{EQ.COR.Wstar_Derivation_Compression_Res_1}, we directly verify $1)$ to $3)$ in Definition \ref{DFN.Wstar_Derivation_Compression_I}.
\end{proof}

\begin{cor}\label{COR.Wstar_Derivation_Projection}
If $p\in L^{\infty}(A,\tau)$ is a projection and $\lset{}p_{k}\rset_{k\in K}\subset\mathcal{A}\cap A_{h}$ a net s.t.~

\begin{itemize}
\item[1)] $p=\bds$-$\lim_{k\in K}p_{k}$,

\item[2)] $p_{k}\in\ker\nabla$ for all $k\in K$,
\end{itemize}

\noindent then $p\in L^{\infty}(A,\tau)_{\nabla}$, $\nabla p=0$, and $\nabla$ is $\lc{}L^{\infty}(A[p],\tau),L^{2}(B[p],\omega)\rc$-compressible.
\end{cor}
\begin{proof}
We use the following results. Let $p\in L^{\infty}(A,\tau)$ be a projection. We know $(\phi,\bpsi,\gamma)$ is $\lc{}L^{\infty}(A[p],\tau),L^{2}(B[p],\omega)\rc$-compressible by Corollary \ref{COR.AF_Cstar_Bimodule_Projection}. The latter shows

\begin{align}\label{EQ.COR.Wstar_Derivation_Projection_1}
\pi_{L^{2}(N,\tau)}^{A}=\pi_{L^{2}(A[p],\tau)}^{A}=L_{p}R_{p},\ \pi_{V}^{B}=\pi_{L^{2}(B[p],\omega)}^{B}=L_{p}^{\phi}R_{p}^{\bpsi}.
\end{align}

\noindent Equation \ref{EQ.COR.Wstar_Derivation_Projection_1} in turn implies

\begin{align}\label{EQ.COR.Wstar_Derivation_Projection_2}
\pi_{L^{2}(N,\tau)}^{A}\lc\mathcal{A}\rc{}=p\mathcal{A}p,\ \pi_{V}^{B}\lc\dom\nabla^{*}\rc{}=p\lc\dom\nabla^{*}\rc{}p.
\end{align}

\noindent Lemma \ref{LEM.Cstar_Trace_Abstract_Projection} shows $p\mathcal{A}p\subset pL^{\infty}(A,\tau)p=L^{\infty}(A[p],\tau)$, as well as $p\mathcal{A}p\subset pL^{2}(A,\tau)p=L^{2}(A[p],\tau)$. We use these inclusions to show $1)$ to $3)$ in Definition \ref{DFN.Wstar_Derivation_Compression_I}.\par
Let $\lset{}p_{k}\rset_{k\in K}\subset\mathcal{A}\cap A_{h}$ be a net s.t.~$1)$ and $2)$ holds. Get $p\in L^{\infty}(A,\tau)_{\nabla}$ and $\nabla p=0$. Thus $1)$ in Proposition \ref{PRP.Wstar_Derivation_BdCon} yields $p\mathcal{A}p\subset L^{\infty}(A,\tau)_{\nabla}$, hence Equation \ref{EQ.COR.Wstar_Derivation_Projection_2} shows $1)$ in Definition \ref{DFN.Wstar_Derivation_Compression_I}. Using density of $\mathcal{A}$ as per $3.2)$ in Definition \ref{DFN.CWstar_Derivation}, Equation \ref{EQ.COR.Wstar_Derivation_Projection_2} further shows $2)$ in Definition \ref{DFN.Wstar_Derivation_Compression_I}. Using $p\in L^{\infty}(A,\tau)_{\nabla}$, Equation \ref{EQ.COR.Wstar_Derivation_Projection_2} and $2)$ in Proposition \ref{PRP.Wstar_Derivation_BdCon}, we directly verify $3)$ in Definition \ref{DFN.Wstar_Derivation_Compression_I} on inner products.
\end{proof}

Definition \ref{DFN.Wstar_Derivation_Compression_II} gives $^{*}$-subalgebras generated by compressions. Proposition \ref{PRP.Wstar_Derivation_Compression_I} lifts properties in Definition \ref{DFN.Wstar_Derivation_Compression_I} to such $^{*}$-subalgebras. The latter therefore serve as domains of compressed symmetric $W^{*}$-derivations. Definition \ref{DFN.Wstar_Derivation_Compression_II} gives compressed symmetric $W^{*}$-derivations. Proposition \ref{PRP.Wstar_Derivation_Compression_II} collects their properties. Notation \ref{NTN.Wstar_Derivation_Compression} fixes conventions.

\begin{dfn}\label{DFN.Wstar_Derivation_Compression_II}
For all compressions $(N,V)$ of $\nabla:\mathcal{A}\longrightarrow L^{2}(B,\omega)$, let $\mathcal{A}_{N}\subset N$ be the $^{*}$-subalgebra generated by $\pi_{L^{2}(N,\tau)}^{A}\lc\mathcal{A}\rc$ in $N$.
\end{dfn}

\begin{rem}
We do not require $^{*}$-subalgebras to be closed in any topology.
\end{rem}

\begin{prp}\label{PRP.Wstar_Derivation_Compression_I}
For all compressions $(N,V)$ of $\nabla:\mathcal{A}\longrightarrow L^{2}(B,\omega)$, we have

\begin{itemize}
\item[1)] $\mathcal{A}_{N}\subset N\cap L^{\infty}(A,\tau)_{\nabla}$ is a $^{*}$-subalgebra,

\item[2)] $\mathcal{A}_{N}\subset N$ is $\sigma$-weakly dense and $\mathcal{A}_{N}\subset L^{2}(N,\tau)$ is $\|.\|_{\tau}$-dense,

\item[3)] $\mathcal{A}_{N}\subset\dom\nabla$, $\pi_{V}^{B}\lc\dom\nabla^{*}\rc\subset\dom\nabla^{*}$, and

\begin{align}\label{EQ.PRP.Wstar_Derivation_Compression_I_1}
\pi_{V}^{B}\lc\nabla x\rc{}=\nabla\pi_{L^{2}(N,\tau)}^{A}(x),\ \pi_{L^{2}(N,\tau)}^{A}\big(\nabla^{*}u\big)=\nabla^{*}\pi_{V}^{B}(u)
\end{align}

\begin{reapply}
\end{reapply}

\noindent for all $x\in\mathcal{A}$ and $u\in\dom\nabla^{*}$.
\end{itemize}
\end{prp}
\begin{proof}
Get $1)$ by $1)$ in Proposition \ref{PRP.Wstar_Derivation_BdCon}. We have $2)$ as H\"older ensures $\mathcal{A}_{N}\subset L^{2}(N,\tau)$. Using $2)$ in Proposition \ref{PRP.Wstar_Derivation_BdCon}, we obtain $3)$ by extending $3)$ in Definition \ref{DFN.Wstar_Derivation_Compression_I}.
\end{proof}

\begin{prp}\label{PRP.Wstar_Derivation_Compression_II}
For all compressions $(N,V)$ of $\nabla:\mathcal{A}\longrightarrow L^{2}(B,\omega)$, we have

\begin{itemize}
\item[1)] $\pi_{V}^{B}\nabla\subset\nabla\pi_{L^{2}(N,\tau)}^{A}$ and $\pi_{L^{2}(N,\tau)}^{A}\nabla^{*}\subset\nabla^{*}\pi_{V}^{B}$,

\item[2)] $\nabla\vert_{\mathcal{A}_{N}}:\mathcal{A}_{N}\longrightarrow V$ is a symmetric $W^{*}$-derivation and

\begin{itemize}
\item[2.1)] $\pi_{L^{2}(N,\tau)}^{A}\lc\dom\nabla\rc{}=L^{2}(N,\tau)\cap\dom\nabla$, \phantom{\big)}

\item[2.2)] $\nabla\vert_{L^{2}(N,\tau)}:L^{2}(N,\tau)\cap\dom\nabla\longrightarrow V$ is $\lc\|.\|_{\tau},\|.\|_{\omega}\rc$-closure of $\nabla\vert_{\mathcal{A}_{N}}$, \phantom{\big)}
\end{itemize}

\begin{reapply}
\end{reapply}

\item[3)] $\big(\nabla\vert_{V}\big)^{*}:V\cap\dom\nabla^{*}\longrightarrow L^{2}(N,\tau)$ is a closed unbounded operator and

\begin{itemize}
\item[3.1)] $\pi_{V}^{B}\lc\dom\nabla^{*}\rc{}=V\cap\dom\nabla^{*}$, \phantom{\big)}

\item[3.2)] $\big(\nabla\vert_{V}\big)^{*}=\big(\nabla\vert_{L^{2}(N,\tau)}\big)^{*}$, \phantom{\big)}
\end{itemize}

\begin{reapply}
\end{reapply}

\item[4)] $\Delta\in\UBII\lc{}L^{2}(A,\tau)\rc_{+}\cap\UBII_{L^{2}(N,\tau)}\lc{}L^{2}(A,\tau)\rc$ and $\Delta\vert_{L^{2}(N,\tau)}=\big(\nabla\vert_{L^{2}(N,\tau)}\big)^{*}\big(\nabla\vert_{L^{2}(N,\tau)}\big)$.
\end{itemize}
\end{prp}
\begin{proof}
Proposition \ref{PRP.Wstar_Derivation_Compression_I} ensures Lemma \ref{LEM.Reducible_Operator} applies to $\nabla:L^{2}(A,\tau)\longrightarrow L^{2}(B,\omega)$ for core $\mathcal{A}$ and $\nabla^{*}:L^{2}(B,\omega)\longrightarrow L^{2}(A,\tau)$ for $\dom\nabla^{*}$. Note $(\phi,\bpsi,\gamma)$ being $(N,V)$-compressible implies $V$ is a symmetric $W^{*}$-bimodule over $N$ as per $2)$ in Proposition \ref{PRP.AF_Cstar_Bimodule_Compression}. Set

\begin{align}\label{EQ.PRP.Wstar_Derivation_Compression_II_1}
A_{N}:=\overline{\mathcal{A}_{N}}^{\|.\|_{A}}=C^{*}\lc\mathcal{A}_{N}\rc{}.
\end{align}

\noindent The second identity in Equation \ref{EQ.PRP.Wstar_Derivation_Compression_II_1} follows from $1)$ in Proposition \ref{PRP.Wstar_Derivation_Compression_I}. Using $2)$ in Proposition \ref{PRP.Wstar_Derivation_Compression_I}, note $A_{N}\subset N$ is a $\sigma$-weakly dense $C^{*}$-subalgebra. Using the Leibniz rule and symmetry, Equation \ref{EQ.PRP.Wstar_Derivation_Compression_I_1} shows $\nabla\lc\mathcal{A}_{N}\rc\subset V$. Thus $2)$ in Proposition \ref{PRP.Wstar_Derivation_BdCon} and $1)$ in Proposition \ref{PRP.Wstar_Derivation_Compression_I} show the restriction $\nabla:\mathcal{A}_{N}\longrightarrow V$ of $\nabla:L^{\infty}(A,\tau)_{\nabla}\longrightarrow L^{2}(B,\omega)$ to $\mathcal{A}_{N}$ is a symmetric $\mathcal{A}_{N}$-module derivation. It satisfies $3.2)$ in Definition \ref{DFN.CWstar_Derivation} by $2)$ in Proposition \ref{PRP.Wstar_Derivation_Compression_I}. Since $\pi_{L^{2}(N,\tau)}^{A}\lc\dom\nabla\rc$ is $\|.\|_{\nabla}$-closure of $\mathcal{A}_{N}$, it satisfies $2.1)$ and in turn $2.2)$ by $3)$ in Proposition \ref{PRP.Wstar_Derivation_Compression_I}. Hence $\nabla:\mathcal{A}_{N}\longrightarrow V$ satisfies $3.3)$ in Definition \ref{DFN.CWstar_Derivation} and is a symmetric $C^{*}$-derivation. We show it is a symmetric $W^{*}$-derivation.\par


\pagebreak


We use the following. We already show and use $1)$ above. Being the restriction to $\mathcal{A}_{N}$ ensures $4.2)$ in Definition \ref{DFN.CWstar_Derivation}. Semi-finiteness of $N$ moreover shows there exists noncommutative conditional expectation from $L^{\infty}(A,\tau)$ to $N$ as per Remark \ref{REM.Cstar_Trace_Abstract_Dualisation}, i.e.~a normal unital bounded linear map

\begin{align}\label{EQ.PRP.Wstar_Derivation_Compression_II_2}
\pi_{N}^{L^{\infty}(A,\tau)}:L^{\infty}(A,\tau)\longrightarrow N
\end{align}

\noindent restricting to $\pi_{L^{2}(N,\tau)}^{A}$ on $L^{2,\infty}(A,\tau)$ and satisfying a trace identity \lc{}cf.~Remark \ref{REM.Wstar_Trace_NCE}\rc{}.\par
Applying the noncommutative conditional expectation to an approximating net for $\nabla:\mathcal{A}\longrightarrow L^{2}(B,\omega)$ as per $4.3)$ in Definition \ref{DFN.CWstar_Derivation} yields one for $\nabla:\mathcal{A}_{N}\longrightarrow V$. This uses $1)$ and restriction of the noncommutative conditional expectation to the Hilbert space projection. Thus $\nabla:\mathcal{A}_{N}\longrightarrow V$ is a symmetric $W^{*}$-derivation, hence $2)$ follows since we have $2.1)$ and $2.2)$. Using $1)$, we directly verify $3.1)$ and $\nabla^{*}\lc{}V\cap\dom\nabla^{*}\rc\subset L^{2}(N,\tau)$. The latter implies

\begin{align}\label{EQ.PRP.Wstar_Derivation_Compression_II_3}
\dom\big(\nabla\vert_{L^{2}(N,\tau)}\big)^{*}=\pi_{V}^{B}\lc\dom\nabla^{*}\rc{}.   
\end{align}

\noindent Equation \ref{EQ.PRP.Wstar_Derivation_Compression_II_3} and $2.2)$ show $3.2)$. Altogether, get $1)$ to $3)$. Note $3)$ implies $4)$.
\end{proof}

\begin{dfn}\label{DFN.Wstar_Derivation_Compression_III}
For all compressions $(N,V)$ of $\nabla:\mathcal{A}\longrightarrow L^{2}(B,\omega)$, set

\begin{itemize}
\item[1)] $\nabla_{\hspace{-0.055cm} N}:=\nabla\vert_{\mathcal{A}_{N}}$,

\item[2)] $\Delta_{N}:=\mathrlap{\phantom{\nabla}_{\hspace{-0.055cm} N}}\nabla^{*}\nabla_{\hspace{-0.055cm} N}$.
\end{itemize}
\end{dfn}

\begin{ntn}\label{NTN.Wstar_Derivation_Compression}
Following Notation \ref{NTN.Ubd_Closure}, we additionally use $\nabla_{\hspace{-0.055cm} N}$ to denote closures in Definition \ref{DFN.Wstar_Derivation_Compression_III} and throughout our discussion. Proposition \ref{PRP.Wstar_Derivation_Compression_II} therefore states $\nabla_{\hspace{-0.055cm} N}=\nabla\vert_{L^{2}(A,\tau)}$, $\mathrlap{\phantom{\nabla}_{\hspace{-0.055cm} N}}\nabla^{*}=\big(\nabla\vert_{V}\big)^{*}$ and $\Delta_{N}=\Delta\vert_{L^{2}(A,\tau)}$.
\end{ntn}

\begin{prp}\label{PRP.Wstar_Derivation_Compression_III}
Let $\nabla:\mathcal{A}\longrightarrow L^{2}(B,\omega)$ be $(N,V)$-compressible.

\begin{itemize}
\item[1)] For all $g\in C_{b}\lc{}[0,\infty)\rc$, we have

\begin{align}\label{EQ.PRP.Wstar_Derivation_Compression_III_1}
g(\Delta)=g\lc\Delta_{N}\rc\oplus g\lc\Delta\vert_{L^{2}(N,\tau)^{\perp}}\rc{}    
\end{align}

\begin{reapply}
\end{reapply}

\noindent w.r.t.~$\BII\lc{}L^{2}(N,\tau)\rc\oplus\BII\lc{}L^{2}(N,\tau)^{\perp}\rc$.

\item[2)] $g\lc\Delta_{N}\rc\in\BII\lc{}L^{2}(N,\tau)\rc\subset\BII_{V}\lc{}L^{2}(A,\tau)\rc$ and $g\lc\Delta_{N}\rc{}=\com_{L^{2}(A,\tau)}g\lc\Delta_{N}\rc$.
\end{itemize}
\end{prp}
\begin{proof}
Note $4)$ in Proposition \ref{PRP.Wstar_Derivation_Compression_II} shows $\Delta\geq 0$ on $L^{2}(A,\tau)$ is $L^{2}(N,\tau)$-reducible. Using the latter, get $1)$ and $2)$ by $2)$ in Corollary \ref{COR.Compression_Preservation_I} because $\Delta_{N}=\Delta\vert_{L^{2}(N,\tau)}$.
\end{proof}


\subsection[Quantum gradients for AF-$C^{*}$-bimodules]{Quantum gradients for AF-$\mathbf{C}^{*}$-bimodules}\label{SSEC.NCDS_NCG_QG}

In the AF-$C^{*}$-setting, Equation \ref{EQ.COR.Wstar_Derivation_Compression_Res_1} provides the sufficient condition for compressing symmetric $W^{*}$-derivations to induced AF-$C^{*}$-bimodules. We therefore define quantum gradients to be symmetric $W^{*}$-derivations s.t.~Equation \ref{EQ.COR.Wstar_Derivation_Compression_Res_1} holds for all $j\in\mathbb{N}$. Their compression is Definition \ref{DFN.Wstar_Derivation_Compression_III}. Compressing to induced AF-$C^{*}$-bimodules and taking limits is finite-dimensional approximation of quantum gradients. Proposition \ref{PRP.Wstar_Derivation_Compression_II} and Proposition \ref{PRP.Wstar_Derivation_Compression_III} imply such compatibility transfers as claimed.\par
Standard constructions of quantum gradients are direct sum, tensor product, as well as internal quantum gradients. We further construct dynamic quantum gradients by weak differentiation of twisted conjugation groups. We include a non-twisted case. In Subsection \ref{SSEC.QOT_DT_BSP}, standard constructions using dynamic quantum gradients provide fundamental example classes. Standard references for unbounded algebra derivations generating $C^{*}$-dynamical systems are \cite{BK.Ped.2018.Cstar_Algebras} and \cite{BK.Sak.1991.OpAlg_Dynamics}. We moreover refer to \cite{BK.Bra.1987.OpAlg_Quantum_StM_I}\cite{BK.Bra.1987.OpAlg_Quantum_StM_II} as comprehensive treatment of $C^{*}$-dynamical systems in quantum statistical mechanics. Standard reference for the weak differentiation of, in general non-twisted, conjugation groups is \cite{ART.Chr.2016.Weakly_Differentiable_Operators}. Their weak derivatives generalise inner derivations \cite{ART.Kad.1966.OpAlg_Derivations}.


\subsubsection*{Definition and properties}

Definition \ref{DFN.Wstar_Derivation_QG} gives quantum gradients. They are symmetric $W^{*}$-derivations by Proposition \ref{PRP.Wstar_Derivation_QG_I}. Proposition \ref{PRP.Wstar_Derivation_QG_I} collects properties. We compress quantum gradients. First, we compress to induced AF-$C^{*}$-bimodules as per Corollary \ref{COR.Wstar_Derivation_Compression_Restriction}. Secondly, we compress with projections as per Corollary \ref{COR.Wstar_Derivation_Projection} assuming additional properties. Using the first one, finite-dimensional approximation is $4)$ in Proposition \ref{PRP.Wstar_Derivation_QG_I}. This is compatibility of quantum gradients with compression and finite-dimensional approximation.\par
Let $(A,\tau)$ and $(B,\omega)$ be tracial AF-$C^{*}$-algebras. Let $(\phi,\bpsi,\gamma)$ be an AF-$A$-bimodule structure on $B$. Note Remark \ref{REM.Wstar_Derivation_QG} concerning closure of quantum gradients.

\begin{dfn}\label{DFN.Wstar_Derivation_QG}
Let $\nabla:A_{0}\longrightarrow L^{2}(B,\omega)$ be a symmetric $A_{0}$-module derivation.

\begin{itemize}
\item[1)] We say that $\nabla$ is a quantum gradient if $B_{0}\subset\dom\nabla^{*}$ and

\begin{align}\label{EQ.DFN.Wstar_Derivation_QG_1}
\nabla(A_{j})\subset B_{j},\ \nabla^{*}(B_{j})\subset A_{j}
\end{align}

\begin{reapply}
\end{reapply}

\noindent for all $j\in\mathbb{N}$. Equation \ref{EQ.DFN.Wstar_Derivation_QG_1} is called locality. We further call $\Delta:=\nabla^{*}\nabla$ a quantum Laplacian.

\item[2)] Let $\nabla$ be a quantum gradient. For all $j\in\mathbb{N}$, we call $\nabla_{\hspace{-0.055cm} j}:=\nabla\vert_{A_{j}}:A_{j}\longrightarrow B_{j}$ the $j$-th restricted quantum gradient and $\Delta_{j}:=\mathrlap{\phantom{\nabla}_{\hspace{-0.055cm} j}}\nabla^{*}\nabla_{\hspace{-0.055cm} j}$ the $j$-th restricted Laplacian.
\end{itemize}
\end{dfn}

\begin{rem}\label{REM.Wstar_Derivation_QG}
Let $\nabla:A_{0}\longrightarrow L^{2}(B,\omega)$ be a quantum gradient. Since $B_{0}\subset L^{2}(B,\omega)$ is $\|.\|_{\omega}$-dense and $B_{0}\subset\dom\nabla^{*}$, we see $\nabla$ is $\lc\|.\|_{\tau},\|.\|_{\omega}\rc$-closable \lc{}cf.~Theorem 5.1.5 in \cite{BK.Ped.1989.Analysis_Now}\rc{}.
\end{rem}


\pagebreak


\begin{prp}\label{PRP.Wstar_Derivation_QG_I}
Let $\nabla:A_{0}\longrightarrow L^{2}(B,\omega)$ be a quantum gradient.

\begin{itemize}
\item[1)] $\nabla:A_{0}\longrightarrow L^{2}(B,\omega)$ is a symmetric $W^{*}$-derivation.

\item[2)] For all $j\leq k$ in $\mathbb{N}$, we have

\begin{itemize}
\item[2.1)] $\pi_{j}^{B}\nabla\subset\nabla\pi_{j}^{A}$ and $\pi_{j}^{A}\nabla^{*}\subset\nabla^{*}\pi_{j}^{B}$, 

\item[2.2)] $\pi_{j}^{B}\nabla\pi_{k}^{A}\subset\nabla\pi_{j}^{A}$ and $\pi_{j}^{A}\nabla^{*}\pi_{k}^{B}\subset\nabla^{*}\pi_{j}^{B}$.
\end{itemize}

\begin{reapply}
\end{reapply}

\item[3)] For all $j\in\mathbb{N}$, we have

\begin{itemize}
\item[3.1)] $\nabla:A_{0}\longrightarrow L^{2}(B,\omega)$ is $(A_{j},B_{j})$-compressible,

\item[3.2)] $\nabla_{\hspace{-0.055cm} j}:A_{j}\longrightarrow B_{j}$ is a quantum gradient, $\mathrlap{\phantom{\nabla}_{\hspace{-0.055cm} j}}\nabla^{*}=\big(\nabla\vert_{B_{j}}\big)^{*}$ and $\Delta_{j}=\Delta_{A_{j}}$.
\end{itemize}

\begin{reapply}
\end{reapply}

\item[4)] We have

\begin{itemize}
\item[4.1)]  $A_{0}$ is core of $\nabla$ and $u=\|.\|_{\nabla}$-$\lim_{j\in\mathbb{N}}\pi_{j}^{A}(u)$ for all $u\in\dom\nabla$,

\item[4.2)] $B_{0}$ is core of $\nabla^{*}$ and $v=\|.\|_{\nabla^{*}}$-$\lim_{j\in\mathbb{N}}\pi_{j}^{B}(v)$ for all $v\in\dom\nabla^{*}$,

\item[4.3)] $A_{0}$ is core of $\Delta$ and $w=\|.\|_{\Delta}$-$\lim_{j\in\mathbb{N}}\pi_{j}^{A}(w)$ for all $w\in\dom\Delta$.
\end{itemize}

\begin{reapply}
\end{reapply}

\item[5)] We have $\gamma\lc\dom\nabla^{*}\rc{}=\dom\nabla^{*}$. For all $u\in\dom\nabla^{*}$, we have $\nabla^{*}\gamma(u)=\big(\nabla^{*}u\big)^{*}$. 
\end{itemize}
\end{prp}
\begin{proof}
Using $\lc\|.\|_{\tau},\|.\|_{\omega}\rc$-closure of $\nabla:A_{0}\longrightarrow L^{2}(B,\omega)$, we have $\nabla:L^{2}(A,\tau)\longrightarrow L^{2}(B,\omega)$ with core $A_{0}\subset\dom\nabla$. We further have $\nabla^{*}:L^{2}(B,\omega)\longrightarrow L^{2}(A,\tau)$ and $B_{0}\subset\dom\nabla^{*}$. Note $A_{0}\subset L^{2}(A,\tau)$ and $B_{0}\subset L^{2}(B,\omega)$ are dense in respective Hilbert space norms.\par
We use $4.1)$ to show $1)$. We see $3)$ in Proposition \ref{PRP.AF_Cstar_Trace_III} shows $4.1)$ if Equation \ref{EQ.PRP.Wstar_Derivation_QG_I_1} holds for all $j\in\mathbb{N}$. We use Lemma \ref{LEM.Reducible_Operator}. Set $H_{0}=A_{j}$ and $H_{1}=B_{j}$ in each case. By testing on the inner product, density in Hilbert space norms and Equation \ref{EQ.DFN.Wstar_Derivation_QG_1} show Lemma \ref{LEM.Reducible_Operator} applies to $\nabla:\dom\nabla\longrightarrow L^{2}(B,\omega)$ using core $A_{0}$. For all $j\in\mathbb{N}$, we have

\begin{align}\label{EQ.PRP.Wstar_Derivation_QG_I_1}
\pi_{j}^{B}\nabla\subset\nabla\pi_{j}^{A}.
\end{align}

\noindent Get $4.1)$. We show $1)$. Note $\nabla:A_{0}\longrightarrow L^{2}(B,\omega)$ is a symmetric $C^{*}$-derivation if it is $\lc\|.\|_{A},\|.\|_{\omega}\rc$-closable. Using $A_{0}\subset L^{\infty}(A,\tau)=L^{1}(A,\tau)^{*}$ and $B_{0}\subset L^{2}(B,\omega)$ $\|.\|_{\omega}$-dense, we directly verify closability. Using $4.1)$, we directly verify $4.2)$ and $4.3)$ in Definition \ref{DFN.CWstar_Derivation}. For $4.3)$ in Definition \ref{DFN.CWstar_Derivation}, we use $\lset{}1_{A_{j}}\rset_{j\in\mathbb{N}}$ as approximating sequence of $1_{A}$. We see $\nabla:A_{0}\longrightarrow L^{2}(B,\omega)$ is a symmetric $W^{*}$-derivation. Get $1)$.\par


\pagebreak


We know $1)$. We therefore have $3)$ by Equation \ref{EQ.DFN.Wstar_Derivation_QG_1}, Corollary \ref{COR.Wstar_Derivation_Compression_Restriction}, as well as $3)$ and $4)$ in Proposition \ref{PRP.Wstar_Derivation_Compression_II}. Equation \ref{EQ.PRP.Wstar_Derivation_QG_I_1} shows all claims for $\nabla$ in $2)$. We directly verify all claims for $\nabla^{*}$ in $2)$ by restricting $\nabla^{*}$ as per $3.2)$. This shows $4.2)$ by $3)$ in Proposition \ref{PRP.AF_Cstar_Trace_III}. Then $4.3)$ follows by $1)$. Altogether, get $1)$ to $4)$.\par
We show $5)$. The anti-linear isometric property of $\gamma$ is $\lgl\gamma(v),\gamma(w)\rgl_{\omega}=\overline{\lgl v,w\rgl}$ for all $v,w\in L^{2}(B,\omega)$. For all $x\in A_{0}$ and $u\in B_{0}$, we apply $\gamma\lc\nabla x\rc{}=\nabla x^{*}$ and the anti-linear isometry property to calculate

\begin{align}\label{EQ.PRP.Wstar_Derivation_QG_I_2}
\lgl\nabla^{*}\gamma(u),x\rgl_{\tau}=\lgl\gamma(u),\nabla x\rgl_{\omega}=\overline{\lgl u,\nabla x^{*}\rgl}_{\omega}=\lgl\big(\nabla^{*}u\big)^{*},x\rgl_{\tau}.
\end{align}

\noindent Using $4)$, Equation \ref{EQ.PRP.Wstar_Derivation_QG_I_2} shows $5)$ by closure.
\end{proof}

\begin{dfn}\label{DFN.Wstar_Derivation_QG_Projection}
Let $\nabla:A_{0}\longrightarrow L^{2}(B,\omega)$ be a quantum gradient.

\begin{itemize}
\item[1)] Let $p\in L^{\infty}(A,\tau)$ be a projection. We say that $\nabla$ is $p$-compressible if the conditions of Corollary \ref{COR.Wstar_Derivation_Projection} are satisfied for $\nabla$ and $p$.

\item[2)] Let $\nabla$ be $p$-compressible and $A_{0,L^{\infty}(A[p],\tau)}$ the $^{*}$-subalgebra generated by $pA_{0}p$ in $L^{\infty}(A[p],\tau)$. We call $\nabla_{p}:=\nabla_{L^{\infty}(A[p],\tau)}:A_{0,L^{\infty}(A[p],\tau)}\longrightarrow L^{2}(B[p],\omega)$ a $p$-compressed quantum gradient and $\Delta_{p}:=\Delta_{L^{\infty}(A[p],\tau)}$ its $p$-compressed Laplacian.
\end{itemize}
\end{dfn}

\begin{prp}\label{PRP.Wstar_Derivation_QG_II}
Let $\nabla:A_{0}\longrightarrow L^{2}(B,\omega)$ be a quantum gradient.

\begin{itemize}
\item[1)] For all $j\in\mathbb{N}$, we have $\Delta_{j}=\Delta_{A_{j}}=\mathrlap{\phantom{\nabla}_{\hspace{-0.055cm} j}}\nabla^{*}\nabla_{\hspace{-0.055cm} j}$.

\item[2)] If $\nabla$ is $p$-compressible, then $\Delta_{p}=\Delta\vert_{L^{2}(A[p],\tau)}=\mathrlap{\phantom{\nabla}_{\hspace{-0.055cm} p}}\nabla^{*}\nabla_{\hspace{-0.055cm} p}$.
\end{itemize}
\end{prp}
\begin{proof}
Apply $4)$ in Proposition \ref{PRP.Wstar_Derivation_Compression_II}.
\end{proof}


\subsubsection*{Standard constructions}

We use just three standard constructions for quantum gradients: direct sums, tensor products and internal products. We collect constructions and properties here for use throughout our discussion. For details on direct sums and tensor products of $C^{*}$-~and $W^{*}$-algebras, we refer to Subsection \ref{SSEC.A_Fnd_CWstar}.

\begin{ntn}\label{NTN.DS}
We use superscripts before subscripts to denote instances of objects whenever possible. If this does not yield suitable notation, in particular to prevent any overload of exponents, then we revert to subscripts even as it may introduce double subscripts. The latter must be clear from context, e.g.~$A_{n,j}$ denotes the $j$-th generating $C^{*}$-subalgebra of an AF-$C^{*}$-algebra $A_{n}$ for $n\in\mathbb{N}$. Let $m\in\mathbb{N}$. For all objects $E$ with direct sums, set $E^{m}:=\oplus_{n=1}^{m}E$. For all direct sums $\oplus_{n=1}^{m}H_{n}$ of Hilbert spaces, let 

\begin{align}
\pi_{k}:\oplus_{n=1}^{m}H_{n}\longrightarrow H_{k}
\end{align}

\noindent be the Hilbert space projection from $\oplus_{n=1}^{m}H_{n}$ to $H_{k}$ for all $k\in\lset{}1,\ldots,m\rset$.
\end{ntn}

Definition \ref{DFN.Wstar_Derivation_QG_DS} gives direct sum AF-$C^{*}$-bimodules and quantum gradients for the following data. Let $m\in\mathbb{N}$ and $(A,\tau)$ be a tracial AF-$C^{*}$-algebra. For all $n\in\lset{}1,\ldots,m\rset$, let $\lc{}B_{n},\omega_{n}\rc$ be a tracial AF-$C^{*}$-algebra and $\lc\phi_{n},\bpsi_{n},\gamma_{n}\rc$ an AF-$A$-bimodule structure on $B_{n}$. We define f.s.n.~trace $\oplus_{n=1}^{m}\omega_{n}$ on $\oplus_{n=1}^{m}L^{\infty}(B_{n},\omega_{n})$ by setting

\begin{align}\label{EQ.SSEC.NCDS_NCG_QG_1}
\lc\oplus_{n=1}^{m}\omega_{n}\rc{}(x):=\sum_{n=1}^{m}\omega_{n}\lc{}x_{n}\rc{}
\end{align}

\noindent for all $x=(x_{1},\ldots,x_{n})\in\oplus_{n=1}^{m}L^{\infty}(B_{n},\omega_{n})_{+}$. Get tracial AF-$C^{*}$-algebra $\lc\oplus_{n=1}^{m}B_{n},\oplus_{n=1}^{m}\omega_{n}\rc$ in $\oplus_{n=1}^{m}L^{\infty}(B_{n},\omega_{n})$ generated by $\lset\oplus_{n=1}^{m}B_{n,j}\rset_{j\in\mathbb{N}}$. Let $p\in\lset{}1,2,\infty\rset$. Equation \ref{EQ.SSEC.NCDS_NCG_QG_1} shows $L^{p}\lc\oplus_{n=1}^{m}B_{n},\oplus_{n=1}^{m}\omega_{n}\rc{}=\oplus_{n=1}^{m}L^{p}(B_{n},\omega_{n})$. We obtain local $^{*}$-homomorphisms

\begin{align}\label{EQ.SSEC.NCDS_NCG_QG_2}
\oplus_{n=1}^{m}\phi_{n},\oplus_{n=1}^{m}\bpsi_{n}:A\longrightarrow\oplus_{n=1}^{m}B_{n}    
\end{align}

\noindent by restricting direct sum $^{*}$-homomorphisms to the diagonal $A\subset A^{m}$. We use direct sum anti-linear isometric involution

\begin{align}\label{EQ.SSEC.NCDS_NCG_QG_3}
\oplus_{n=1}^{m}\gamma_{n}:L^{2}\lc\oplus_{n=1}^{m}B_{n},\oplus_{n=1}^{m}\omega_{n}\rc\longrightarrow L^{2}\lc\oplus_{n=1}^{m}B_{n},\oplus_{n=1}^{m}\omega_{n}\rc{}.
\end{align}

\begin{prp}\label{PRP.Wstar_Derivation_QG_DS_I}
Let $m\in\mathbb{N}$ and $(A,\tau)$ tracial AF-$C^{*}$-algebra. For all $n\in\lset{}1,\ldots,m\rset$, let\linebreak $(B_{n},\omega_{n})$ be a tracial AF-$C^{*}$-algebra, $\lc\phi_{n},\bpsi_{n},\gamma_{n}\rc$ an AF-$A$-bimodule structure on $B_{n}$, and $\partial_{n}:A_{0}\longrightarrow L^{2}(B_{n},\omega_{n})$ a quantum gradient. We have

\begin{itemize}
\item[1)] AF-$A$-bimodule structure $\lc\oplus_{n=1}^{m}\phi_{n},\oplus_{n=1}^{m}\bpsi_{n},\oplus_{n=1}^{m}\gamma_{n}\rc$ on $\oplus_{n=1}^{m}B_{n}$,

\item[2)] quantum gradient $\nabla^{\oplus}:=\oplus_{n=1}^{m}\partial_{n}:A_{0}\longrightarrow L^{2}\lc\oplus_{n=1}^{m}B_{n},\oplus_{n=1}^{m}\omega_{n}\rc$ defined by

\begin{align}\label{EQ.PRP.Wstar_Derivation_QG_DS_I_1}
\nabla^{\oplus} x:=\lc\partial_{1} x,\ldots,\partial_{n}x\rc{}
\end{align}

\begin{reapply}
\end{reapply}

\noindent for all $x\in A_{0}$,

\item[3)] $\nabla^{\oplus,*}:=\lc\oplus_{n=1}^{m}\partial_{n}\rc^{*}=\sum_{n=1}^{m}\mathrlap{\phantom{\partial}^{*}}\partial_{n}\pi_{n}$ with core $\oplus_{n=1}^{m}B_{n,0}$ and given by

\begin{align}\label{EQ.PRP.Wstar_Derivation_QG_DS_I_2}
\nabla^{\oplus,*} u=\sum_{n=1}^{m}\mathrlap{\phantom{\partial}^{*}}\partial_{n}u_{n}
\end{align}

\begin{reapply}
\end{reapply}

\noindent for all $u=(u_{1},\ldots,u_{m})\in\dom\nabla^{\oplus,*}=\oplus_{n=1}^{m}\dom\mathrlap{\phantom{\partial}^{*}}\partial_{n}$,

\item[4)] $\Delta^{\oplus}:=\nabla^{\oplus,*}\nabla^{\oplus}=\sum_{n=1}^{m}\mathrlap{\phantom{\partial}^{*}}\partial_{n}\partial_{n}$ with core $A_{0}$ and given by

\begin{align}\label{EQ.PRP.Wstar_Derivation_QG_DS_I_3}
\Delta^{\oplus} u=\sum_{n=1}^{m}\mathrlap{\phantom{\partial}^{*}}\partial_{n}\partial_{n} u 
\end{align}

\begin{reapply}
\end{reapply}

\noindent for all $u\in\dom\Delta^{\oplus}=\bigcap_{n=1}^{m}\dom\mathrlap{\phantom{\partial}^{*}}\partial_{n}\partial_{n}$.
\end{itemize}
\end{prp}
\begin{proof}
Get $1)$ by construction. Using direct sum construction, Equation \ref{EQ.PRP.Wstar_Derivation_QG_DS_I_1} shows $2)$ and Equation \ref{EQ.PRP.Wstar_Derivation_QG_DS_I_2} by reducing to summands. Equation \ref{EQ.PRP.Wstar_Derivation_QG_DS_I_1} and Equation \ref{EQ.PRP.Wstar_Derivation_QG_DS_I_2} thus imply $3)$ and $4)$ by Proposition \ref{PRP.Wstar_Derivation_QG_I}.
\end{proof}

\begin{prp}\label{PRP.Wstar_Derivation_QG_DS_II}
Assume the setting of Proposition \ref{PRP.Wstar_Derivation_QG_DS_I}. Let $f$ be a representing function and $\theta\in [0,1]$. For all $\mu,\eta\in A_{+}^{*}$ and $w\in B^{*}$, we have

\begin{align}\label{EQ.PRP.Wstar_Derivation_QG_DS_II_2}
\mathcal{I}^{f,\theta}(\mu,\eta,w)=\sum_{n=1}^{m}\mathcal{I}_{A,B_{n}}^{f,\theta}\lc\mu,\eta,w\vert_{B_{n}}\rc{}.
\end{align}
\end{prp}
\begin{proof}
We reduce to the finite-dimensional setting by $3)$ in Theorem \ref{THM.QE_AF}. Note $\mathcal{I}^{f,\theta}$ and each $\mathcal{I}_{A,B_{n}}^{f,\theta}$ are l.s.c.~in $w^{*}$-topology by $1)$ in Theorem \ref{THM.QE_AF}. L.s.c.~in $w^{*}$-topology shows Equation \ref{EQ.PRP.Wstar_Derivation_QG_DS_II_2} if it holds for all $\mu,\eta\in A_{+}^{*}$ s.t.~$\sharp\mu,\sharp\eta>0$ in $A$. Equation \ref{EQ.PRP.Wstar_Derivation_QG_DS_II_2} itself follows by construction of noncommutative division operators in this case.
\end{proof}

\begin{dfn}\label{DFN.Wstar_Derivation_QG_DS}
Assume the setting of Proposition \ref{PRP.Wstar_Derivation_QG_DS_I}. We call

\begin{itemize}
\item[1)] $\lc\oplus_{n=1}^{m}\phi_{n},\oplus_{n=1}^{m}\bpsi_{n},\oplus_{n=1}^{m}\gamma_{n}\rc$ their direct sum AF-$C^{*}$-bimodule,

\item[2)] $\nabla^{\oplus}:A_{0}\longrightarrow L^{2}\lc\oplus_{n=1}^{m}B_{n},\oplus_{n=1}^{m}\omega_{n}\rc$ their direct sum quantum gradient,

\item[3)] $\partial_{n}$ the $n$-th partial gradient, $\mathrlap{\phantom{\partial}^{*}}\partial_{n}$ the $n$-th partial adjoint, and $\Delta_{n}:=\mathrlap{\phantom{\partial}^{*}}\partial_{n}\partial_{n}$ the $n$-th Laplacian for all $n\in\lset{}1,\ldots,m\rset$.
\end{itemize}
\end{dfn}

Definition \ref{DFN.Wstar_Derivation_QG_TP} gives tensor product AF-$C^{*}$-bimodules and quantum gradients for the following data. For all $n\in\lset{}1,2\rset$, let $(A_{n},\tau_{n})$ and $(B_{n},\omega_{n})$ be tracial AF-$C^{*}$-algebras with $(\phi_{n},\bpsi_{n},\gamma_{n})$ an AF-$A_{n}$-bimodule structure on $B_{n}$. We determine f.s.n.~trace $\tau_{1}\otimes\tau_{2}$ on $L^{\infty}(A_{1},\tau_{1})\otimes L^{\infty}(A_{2},\tau_{2})$ by setting

\begin{align}\label{EQ.SSEC.NCDS_NCG_QG_4}
\big(\tau_{1}\otimes\tau_{2}\big)\lc{}x\otimes y\rc{}:=\tau_{1}(x)\tau_{2}(y)
\end{align}

\noindent for all $x\in\mathfrak{m}_{\tau_{1}}$ and $y\in\mathfrak{m}_{\tau_{2}}$. Note both $A_{1}$ and $A_{2}$ are nuclear. Get tracial AF-$C^{*}$-algebra $(A_{1}\otimes A_{2},\tau_{1}\otimes\tau_{2})$ in $L^{\infty}(A_{1},\tau_{1})\otimes L^{\infty}(A_{2},\tau_{2})$ generated by $\lset{}A_{1,j}\otimes A_{2,j}\rset_{j\in\mathbb{N}}$.\par
Let $p\in\lset{}2,\infty\rset$. Equation \ref{EQ.SSEC.NCDS_NCG_QG_4} shows 

\begin{align}\label{EQ.SSEC.NCDS_NCG_QG_EXTRA_1}
L^{p}(A_{1}\otimes A_{2},\tau_{1}\otimes\tau_{2})=L^{p}(A_{1},\tau_{1})\otimes L^{p}(A_{2},\tau_{2}). 
\end{align}

\noindent Note the above construction likewise yields tracial AF-$C^{*}$-algebra $(B_{1}\otimes B_{2},\omega_{1}\otimes\omega_{2})$ in $L^{\infty}(B_{1},\omega_{1})\otimes L^{\infty}(B_{2},\omega_{2})$ generated by $\lset{}B_{1,j}\otimes B_{2,j}\rset_{j\in\mathbb{N}}$ s.t.~Equation \ref{EQ.SSEC.NCDS_NCG_QG_EXTRA_1} is

\begin{align}\label{EQ.SSEC.NCDS_NCG_QG_EXTRA_2}
L^{p}(B_{1}\otimes B_{2},\omega_{1}\otimes\omega_{2})=L^{p}(B_{1},\omega_{1})\otimes L^{p}(B_{2},\omega_{2})
\end{align}

\noindent in each case.\par
Note Corollary \ref{COR.Wstar_TP} shows we obtain local $^{*}$-homomorphisms

\begin{align}\label{EQ.SSEC.NCDS_NCG_QG_5}
\phi_{1}\otimes\phi_{2},\bpsi_{1}\otimes\bpsi_{2}:A_{1}\otimes A_{2}\longrightarrow B_{1}\otimes B_{2}
\end{align}

\noindent by restricting tensored $^{*}$-homomorphisms to $A_{1}\otimes A_{2}\subset L^{\infty}(A_{1}\otimes A_{2},\tau_{1}\otimes\tau_{2})$. We use tensor product anti-linear isometric involution 

\begin{align}\label{EQ.SSEC.NCDS_NCG_QG_6}
\gamma_{1}\otimes\gamma_{2}:L^{2}(B_{1}\otimes B_{2},\omega_{1}\otimes\omega_{2})\longrightarrow L^{2}(B_{1}\otimes B_{2},\omega_{1}\otimes\omega_{2}).
\end{align}

\begin{prp}\label{PRP.Wstar_Derivation_QG_TP}
For all $n\in\lset{}1,2\rset$, let $(A_{n},\tau_{n})$ and $(B_{n},\omega_{n})$ be tracial AF-$C^{*}$-algebras with $(\phi_{n},\bpsi_{n},\gamma_{n})$ an AF-$A_{n}$-bimodule structure on $B_{n}$, as well as $\delta_{n}:A_{n,0}\longrightarrow L^{2}(B_{n},\omega_{n})$ a quantum gradient. We have

\begin{itemize}
\item[1)] AF-$A_{1}\otimes A_{2}$-bimodule structure $\lc\phi_{1}\otimes\phi_{2},\bpsi_{1}\otimes\bpsi_{2},\gamma_{1}\otimes\gamma_{2}\rc$ on $B_{1}\otimes B_{2}$,

\item[2)] quantum gradient $\nabla^{\otimes}:A_{1,0}\odot A_{2,0}\longrightarrow L^{2}(B_{1}\otimes B_{2},\omega_{1}\otimes\omega_{2})$ defined by

\begin{align}\label{EQ.PRP.Wstar_Derivation_QG_TP_1}
\nabla^{\otimes}x\otimes y:=\delta_{1} x\otimes\bpsi_{2}(y)+\phi_{1}(x)\otimes\delta_{2}y
\end{align}

\begin{reapply}
\end{reapply}

\noindent for all $x\in A_{1,0}$ and $y\in A_{2,0}$,

\item[3)] $\nabla^{\otimes,*}:=\big(\nabla^{\otimes}\big)^{*}$ with core $\dom\delta_{1}^{*}\odot\dom\delta_{2}^{*}$ and determined by

\begin{align}\label{EQ.PRP.Wstar_Derivation_QG_TP_2}
\nabla^{\otimes,*}u\otimes v=\delta_{1}^{*}u\otimes\bpsi_{2}^{*}(v)+\phi_{1}^{*}(u)\otimes\delta_{2}^{*}v
\end{align}

\begin{reapply}
\end{reapply}

\noindent for all $u\in\dom\delta_{1}^{*}$ and $v\in\dom\delta_{2}^{*}$.
\end{itemize}
\end{prp}
\begin{proof}
Get $1)$ by construction. Using tensor product construction, Equation \ref{EQ.PRP.Wstar_Derivation_QG_TP_1} shows $\nabla^{\otimes}$ is a symmetric $A_{1,0}\odot A_{2,0}$-module derivation, implies Equation \ref{EQ.PRP.Wstar_Derivation_QG_TP_2}, and yields

\begin{align}\label{EQ.PRP.Wstar_Derivation_QG_TP_3}
B_{1,0}\otimes B_{2,0}\subset\dom\delta_{1}^{*}\odot\dom\delta_{2}^{*}\subset\dom\nabla^{\otimes,*}
\end{align}

\noindent by reducing to elementary tensors. Using inclusions in Equation \ref{EQ.PRP.Wstar_Derivation_QG_TP_3}, Equation \ref{EQ.PRP.Wstar_Derivation_QG_TP_1} and Equation \ref{EQ.PRP.Wstar_Derivation_QG_TP_2} imply locality. Proposition \ref{PRP.Wstar_Derivation_QG_I} implies all remaining claims.
\end{proof}

\begin{dfn}\label{DFN.Wstar_Derivation_QG_TP}
Assume the setting of Proposition \ref{PRP.Wstar_Derivation_QG_TP}. We call

\begin{itemize}
\item[1)] $\lc\phi_{1}\otimes\phi_{2},\bpsi_{1}\otimes\bpsi_{2},\gamma_{1}\otimes\gamma_{2}\rc$ their tensor product AF-$C^{*}$-bimodule,

\item[2)] $\nabla^{\otimes}:A_{1,0}\odot A_{2,0}\longrightarrow L^{2}(B_{1}\otimes B_{2},\omega_{1}\otimes\omega_{2})$ their tensor product quantum gradient.
\end{itemize}
\end{dfn}

\begin{rem}\label{REM.Wstar_Derivation_QG_TP}
Tensor product quantum gradients have Laplacians with mixed terms coupling their factors. This follows by construction. It differs from the decomposition given by Equation \ref{EQ.PRP.Wstar_Derivation_QG_DS_I_3} for direct sum quantum gradients.
\end{rem}

Definition \ref{DFN.CWstar_Derivation_Generalised_Discrete} gives generalised discrete derivatives. Example \ref{BSP.QOT_KL_Parametrised} shows these specialise to discrete derivatives and internal quantum gradients. Let $(A,\tau)$ be unital tracial $C^{*}$-algebra in $M$ s.t.~$\tau<\infty$. We have $^{*}$-homomorphisms 

\begin{align}\label{EQ.SSEC.NCDS_NCG_QG_7}
\phi^{\Int}:=\restr{1.075}{\textrm{id}_{A\otimes A}}{A\otimes\langle 1_{A}\rangle_{\mathbb{C}}},\bpsi^{\Int}:=\restr{1.075}{\textrm{id}_{A\otimes A}}{\langle 1_{A}\rangle_{\mathbb{C}}\otimes A}:A\longrightarrow A\otimes A
\end{align}

\noindent given by restriction to unital $C^{*}$-subalgebras $A\cong A\otimes\langle 1_{A}\rangle_{\mathbb{C}}\cong\langle 1_{A}\rangle_{\mathbb{C}}\otimes A$ of $A\otimes A$. We have f.s.n.~trace $\tau\otimes\tau$ on $M\otimes M$ and unital tracial $C^{*}$-algebra $A\otimes A$ in $M\otimes M$. Thus $L^{2}(A\otimes A,\tau\otimes\tau)$ is a symmetric $C^{*}$-bimodule over $A$ equipped with $\lc{}L_{M}\circ\phi,R_{M}\circ\bpsi\rc$-action and $\gamma=\Adj$ as per Example \ref{BSP.AF_Cstar_Bimodule_Canonical} for $A\otimes A$ as anti-linear involution.

\begin{dfn}\label{DFN.CWstar_Derivation_Generalised_Discrete}
Let $(A,\tau)$ be unital tracial $C^{*}$-algebra in $M$ s.t.~$\tau<\infty$. We define symmetric $C^{*}$-bimodule $L^{2}(A\otimes A,\tau\otimes\tau)$ over $A$ by Equation \ref{EQ.SSEC.NCDS_NCG_QG_7}. We define generalised discrete derivative $\delta:A\longrightarrow L^{2}(A\otimes A,\tau\otimes\tau)$ on $A$ by setting

\begin{align}\label{EQ.DFN.CWstar_Derivation_Generalised_Discrete_1}
\delta x:=x\otimes 1_{A}-1_{A}\otimes x
\end{align}

\noindent for all $x\in A$.
\end{dfn}

\begin{prp}\label{PRP.CWstar_Derivation_Generalised_Discrete}
If $(A,\tau)$ is a unital tracial $C^{*}$-algebra in $M$ s.t.~$\tau<\infty$, then $\delta$ as per Definition \ref{DFN.CWstar_Derivation_Generalised_Discrete} is a bounded symmetric $A$-module derivation.
\end{prp}
\begin{proof}
Note $A\otimes A\subset L^{2}(A\otimes A,\tau\otimes\tau)$ since $\tau\otimes\tau<\infty$. On $A\otimes A$, the symmetric $C^{*}$-bimodule action reduces to left-~and right-algebra multiplication in $A\otimes A$ using $^{*}$-homomorphisms given by Equation \ref{EQ.SSEC.NCDS_NCG_QG_7} as in Equation \ref{EQ.SSEC.NCDS_AF_BIM_4}. For all $x,y\in A$, we use $\delta x,\delta y\in A\otimes A$ and Equation \ref{EQ.DFN.CWstar_Derivation_Generalised_Discrete_1} to calculate $\delta xy=\lc{}x\otimes 1_{A}\rc\lc{}y\otimes 1_{A}-1_{A}\otimes y\rc{}+\lc{}x\otimes 1_{A}-1_{A}\otimes x\rc\lc{}1_{A}\otimes y\rc{}=x\lc\delta y\rc{}+\lc\delta x\rc{} y$. Thus $\delta$ satisfies the Leibniz rule. Symmetry follows at once.
\end{proof}

Let $(A,\tau)$ be a strongly unital tracial AF-$C^{*}$-algebra s.t.~$\tau<\infty$. We equip $A$ with its canonical AF-$A$-bimodule structure. Then $(A\otimes A,\tau\otimes\tau)$ has canonical AF-$A\otimes A$-bimodule structure and the $^{*}$-homomorphisms given by Equation \ref{EQ.SSEC.NCDS_NCG_QG_7} are local.

\begin{prp}\label{PRP.Wstar_Derivation_QG_Internal}
Let $(A,\tau)$ be a strongly unital tracial AF-$C^{*}$-algebra s.t.~$\tau<\infty$. For all $\lambda\geq 0$, we have

\begin{itemize}
\item[1)] AF-$A$-bimodule structure $\lc\phi^{\Int},\bpsi^{\Int},\Adj\rc$ on $A\otimes A$,

\item[2)] bounded quantum gradient $\nabla:A_{0}\longrightarrow L^{2}(A\otimes A,\tau\otimes\tau)$ defined by

\begin{align}\label{EQ.PRP.Wstar_Derivation_QG_Internal_1}
\nabla^{\lambda}x:=\sqrt{\frac{\lambda}{2\tau(1_{A})}}\cdot \big(x\otimes 1_{A}-1_{A}\otimes x\big)
\end{align}

\begin{reapply}
\end{reapply}

\noindent for all $x\in A_{0}$,

\item[3)] $\nabla^{\lambda,*}:=\big(\nabla^{\lambda}\big)^{*}$ bounded and determined by

\begin{align}\label{EQ.PRP.Wstar_Derivation_QG_Internal_2}
\nabla^{\lambda,*}y\otimes z=\sqrt{\frac{\lambda}{2\tau(1_{A})}}\cdot \bigg(\lgl 1_{A},z\rgl_{\tau}y-\lgl 1_{A},y\rgl_{\tau}z\bigg)
\end{align}

\begin{reapply}
\end{reapply}

\noindent for all $y,z\in A$,

\item[4)] $\Delta^{\lambda}:=\nabla^{\lambda,*}\nabla^{\lambda}=\lambda\pi_{\ker\tau}^{A}$.
\end{itemize}
\end{prp}
\begin{proof}
We have $1)$ by construction. Set $\lambda:=2\tau(1_{A})$ without loss of generality. We then suppress $\lambda$ in the superscript. We show $2)$. Note $\nabla=\delta\vert_{A_{0}}$. Proposition \ref{PRP.CWstar_Derivation_Generalised_Discrete} shows it is a symmetric $A_{0}$-module derivation. Using Equation \ref{EQ.PRP.Wstar_Derivation_QG_Internal_1}, we directly verify it is a quantum gradient. Equation \ref{EQ.PRP.Wstar_Derivation_QG_Internal_1} further shows boundedness upon closure. Get $2)$.\par
We show $3)$. Boundedness of $\nabla$ implies $\nabla^{*}$ is bounded and determined on elementary tensors. For all $x\in A_{0}$ and $y,z\in A$, Equation \ref{EQ.PRP.Wstar_Derivation_QG_Internal_1} lets us calculate 

\begin{align}\label{EQ.PRP.Wstar_Derivation_QG_Internal_3}
\lgl\nabla x,y\otimes z\rgl_{\tau\otimes\tau}=\lgl x,\lgl 1_{A},z\rgl_{\tau}y\rgl_{\tau}-\lgl x,\lgl 1_{A},y\rgl_{\tau}z\rgl_{\tau}=\lgl x,\lgl 1_{A},z\rgl_{\tau}y-\lgl 1_{A},y\rgl_{\tau}z\rgl_{\tau}.
\end{align}

\noindent Equation \ref{EQ.PRP.Wstar_Derivation_QG_Internal_3} shows Equation \ref{EQ.PRP.Wstar_Derivation_QG_Internal_2}. Get $3)$. We show $4)$. For all $x\in A_{0}$, Equation \ref{EQ.PRP.Wstar_Derivation_QG_Internal_1} and Equation \ref{EQ.PRP.Wstar_Derivation_QG_Internal_2} show $\Delta x=2\tau(1_{A})\pi_{\ker\tau}^{A}(x)$. Get $4)$ by boundedness.
\end{proof}

\begin{dfn}\label{DFN.Wstar_Derivation_QG_Internal}
Let $(A,\tau)$ be a strongly unital tracial AF-$C^{*}$-algebra s.t.~$\tau<\infty$. For all $\lambda\geq 0$, we call

\begin{itemize}
\item[1)] $\lc\phi^{\Int},\bpsi^{\Int},\Adj\rc$ the internal AF-$A$-bimodule structure on $A\otimes A$,

\item[2)] $\nabla^{\lambda}:A_{0}\longrightarrow L^{2}(A\otimes A,\tau\otimes\tau)$ the $\lambda$-internal quantum gradient on $A$.
\end{itemize}
\end{dfn}


\subsubsection*{Dynamic quantum gradients}

Definition \ref{DFN.Wstar_Derivation_QG_Dynamic_Generator} gives sufficient conditions to construct quantum gradients by weak differentiation of conjugation groups twisted with self-adjoint involutive local $^{*}$-homomorphisms. These dynamic quantum gradients are either twisted or non-twisted. Generators on Hilbert spaces control weak derivatives as per Equation \ref{EQ.SSEC.NCDS_NCG_QG_8} and Equation \ref{EQ.SSEC.NCDS_NCG_QG_9}. We pull back along canonical left-actions upon weak differentiation in our construction, and twist as per Remark \ref{REM.AF_Cstar_Local_Hom_Self_Adjoint}. We use one-parameter semigroups of bounded operators on Banach spaces \cite{BK.Eng_Nag.2000.Semigroups}.\par
Definition \ref{DFN.Wstar_Derivation_QG_Dynamic} gives dynamic quantum gradients. We give two classes of dynamic quantum gradient. First, we consider trace-preserving local $C^{*}$-dynamical systems in Corollary \ref{COR.Wstar_Derivation_QG_Dynamic_System}. Secondly, we consider intertwining self-adjoint unbounded operators generating suitable conjugation groups in Corollary \ref{COR.Wstar_Derivation_QG_Intertwining_I}. Whereas Corollary \ref{COR.Wstar_Derivation_QG_Dynamic_System} yields only non-twisted examples, Corollary \ref{COR.Wstar_Derivation_QG_Intertwining_I} yields both twisted and non-twisted ones. In Subsection \ref{SSEC.QOT_DT_BSP}, standard constructions using dynamic quantum gradients provide fundamental example classes. Those using tracial AF-$C^{*}$-algebras generating hyperfinite factors of type I and II by $\sigma$-weak closure are of particular interest.

\begin{dfn}\label{DFN.Inner_Semigroup}
Let $H$ be a Hilbert space and $\DII\in\UBII(H)_{h}$. We define conjugation group $\Ad^{\mathcal{D}}:\mathbb{R}\longrightarrow\mathrm{Aut}\lc\BII(H)\rc$ of $\DII$ by setting 

\begin{align}\label{EQ.DFN.Inner_Semigroup_1}
\textrm{Ad}_{t}^{\mathcal{D}}(S):=e^{it\DII}Se^{-it\DII}
\end{align}

\noindent for all $t\in\mathbb{R}$ and $S\in\BII(H)$.
\end{dfn}

Let $H$ be a Hilbert space and $\DII\in\UBII(H)_{h}$. For all $S\in\BII(H)$, the weak derivative

\begin{align}\label{EQ.SSEC.NCDS_NCG_QG_8}
\restr{0.925}{\frac{d}{dt}}{t=0,\w}\textrm{Ad}_{t}^{\mathcal{D}}(S)=\w\textrm{-}\lim_{t\rightarrow 0}\hspace{0.025cm} t^{-1}\lc\textrm{Ad}_{t}^{\mathcal{D}}(S)-S\rc{}
\end{align}

\noindent exists if and only if the following two conditions are satisfied \lc{}cf.~Theorem 3.8 in \cite{ART.Chr.2016.Weakly_Differentiable_Operators}\rc{}. First, $S(u)\in\dom\DII$ for all $u\in\dom\DII$. Secondly, that $\DII S-S\DII\in\UBII(H)$ is bounded and closable. Then $\dom\DII S-S\DII=\dom\DII$ and Equation \ref{EQ.SSEC.NCDS_NCG_QG_8} is 

\begin{align}\label{EQ.SSEC.NCDS_NCG_QG_9}
\restr{0.925}{\frac{d}{dt}}{t=0,\w}\textrm{Ad}_{t}^{\mathcal{D}}(S)=\overline{i\lc\DII S-S\DII\rc{}}\in\BII(H).
\end{align}

\noindent Note $\DII S-S\DII$ is bounded, but not a bounded operator in general. This necessitates closure. For conjugation groups, differentiation in weak and strong operator topologies is equivalent \cite{ART.Chr.2016.Weakly_Differentiable_Operators}. We use strong limits in Equation \ref{EQ.SSEC.NCDS_NCG_QG_8} without loss of generality. If $\DII\in\BII(H)$, then $\restr{0.925}{\frac{d}{dt}}{t=0,\w}\Ad_{t}^{\mathcal{D}}(S)=i\lb\DII,S\rb$ for all $S\in\BII(H)$. In fact, all bounded module derivations on $W^{*}$-algebras are inner \lc{}cf.~Theorem XI.3.5 in \cite{BK.Tak.2003.OpAlg_II}\rc{}. This explains use of unbounded module derivations, resp.~non-canonical AF-$C^{*}$-bimodule structures.

\begin{dfn}\label{DFN.AF_Cstar_Local_Hom_Self_Adjoint}
Let $(A,\tau)$ be a tracial AF-$C^{*}$-algebra. A local $^{*}$-homomorphism \linebreak $\phi:A\longrightarrow A$ is self-adjoint if $\phi\in\BII\lc{}L^{2}(A,\tau)\rc_{h}$ as per Definition \ref{DFN.AF_Cstar_Local_Hom_L2}.
\end{dfn}

\begin{rem}\label{REM.AF_Cstar_Local_Hom_Self_Adjoint}
Let $(A,\tau)$ be a tracial AF-$C^{*}$-algebra and $\phi:A\longrightarrow A$ a self-adjoint involutive local $^{*}$-homomorphism. Thus its $L^{2}$-extension is self-adjoint by hypothesis and involutive since $A_{0}\subset L^{2}(A,\tau)$ is $\|.\|_{\tau}$-dense, hence $\phi\in\UII\lc\BII\lc{}L^{2}(A,\tau)\rc\rc$ and $\phi^{\dagger}$ as per Definition \ref{DFN.Unbd_Twist}. For all $T\in\UBII\lc{}L^{2}(A,\tau)\rc$, get $\phi^{\dagger}(T)=\phi T\phi$. For all $x\in L^{0}(A,\tau)$, we have $\phi^{\dagger}(L_{x})=L_{\phi(x)}$ using canonical AF-$C^{*}$-bimodule action on $A$. We obtain

\begin{align}\label{EQ.REM.AF_Cstar_Local_Hom_Self_Adjoint_1}
\phi^{\dagger}\circ L=L\circ\phi.    
\end{align}

\noindent If $T\in L(A)$, then $\phi^{\dagger}(T)\in L(A)$. If $T\in \phi L(A)$, then $T\phi\in L(A)$.
\end{rem}

Let $\AII$ be a $^{*}$-algebra. For all $x,y\in\AII$, we use their commutator $[x,y]=xy-yx$ and anti-commutator $\{x,y\}=xy+yx$ throughout our discussion. Definition \ref{DFN.Wstar_Derivation_QG_Dynamic_Generator} gives twisted commutators and anti-commutators in an unbounded case.\par


\newpage


\begin{dfn}\label{DFN.Wstar_Derivation_QG_Dynamic_Generator}
Let $(A,\tau)$ be a tracial AF-$C^{*}$-algebra and $\phi:A\longrightarrow A$ a self-adjoint involutive local $^{*}$-homomorphism. Note $\phi=\phi^{2}\in\BII(L^{2}(A,\tau))$. Let $\DII\in\UBII\lc{}L^{2}(A,\tau)\rc_{h}$.

\begin{itemize}
\item[1)] We call $\lc\DII,\phi\rc$ generator of a dynamic quantum gradient if

\begin{itemize}
\item[1.1)] for all $j\in\mathbb{N}$ and $x\in A_{j}$, we have

\begin{itemize}
\item[1.1.1)] $\restr{0.925}{\frac{d}{dt}}{t=0,\w}\Ad_{t}^{\mathcal{D}}(L_{x})=\overline{i\lc\DII L_{x}-L_{x}\DII\rc{}}\in\phi L(A_{j})$, \phantom{\vstretch{0.875}{\bigg)}}

\item[1.1.2)] $\restr{0.925}{\frac{d}{dt}}{t=0,\w}\Ad_{-t}^{\mathcal{D}}\lc L_{x}\phi\rc=\overline{i\lc L_{x}\phi\DII-\DII L_{x}\phi\rc}\in L(A_{j})$, \phantom{\vstretch{0.875}{\bigg)}}
\end{itemize}

\begin{reapply}
\end{reapply}

\item[1.2)] for all $x,y\in A_{0}$, we have

\begin{align}\label{EQ.DFN.Wstar_Derivation_QG_Dynamic_Generator_1}
\bigg\langle x,L^{-1}\lc\restr{0.925}{\frac{d}{dt}}{t=0,\w}\textrm{Ad}_{t}^{\mathcal{D}}(L_{y})\phi\rc\bigg\rangle_{\tau}=\bigg\langle L^{-1}\lc\restr{0.925}{\frac{d}{dt}}{t=0,\w}\textrm{Ad}_{-t}^{\mathcal{D}}\lc L_{x}\phi\rc\rc{},y\bigg\rangle_{\tau}.
\end{align}

\begin{reapply}
\end{reapply}

\end{itemize}

\begin{reapply}
\end{reapply}

\item[2)] Let $\lc\DII,\phi\rc$ be generator of a dynamic quantum gradient. We define $\phi$-twisted commutator $\lb\DII,\blank\rb_{A}^{\phi}:A_{0}\longrightarrow L^{2}(A,\tau)$ and anti-commutator $\lset\DII,\blank\rset_{A}^{\phi}:A_{0}\longrightarrow L^{2}(A,\tau)$ by setting

\begin{align}\label{EQ.DFN.Wstar_Derivation_QG_Dynamic_Generator_2}
[\DII,x]_{A}^{\phi}:=L^{-1}\lc\lc\overline{\DII L_{\phi(x)}-L_{\phi(x)}\DII}\rc\phi\rc{},\ \{\DII,x\}_{A}^{\phi}:=L^{-1}\lc\phi\lc\overline{L_{x}\phi\DII-\DII L_{x}\phi}\rc\phi\rc{}
\end{align}

\begin{reapply}
\end{reapply}

\noindent for all $x\in A_{0}$.
\end{itemize}
\end{dfn}

\begin{rem}\label{REM.Wstar_Derivation_QG_Dynamic_Generator}
If $\phi=\id_{A}$, then Equation \ref{EQ.DFN.Wstar_Derivation_QG_Dynamic_Generator_2} reduces to commutators and their negatives. Using $2)$ in Lemma \ref{LEM.Wstar_Derivation_QG_Intertwining}, we see non-trivial $\phi$ as per Example \ref{BSP.QOT_Type_II_Twisted} yield anti-commutators up to twist generalising \cite{ART.Car_Maa.2014.Quantum_OT_I}.
\end{rem}

Let $(A,\tau)$ be a tracial AF-$C^{*}$-algebra. Let $\phi:A\longrightarrow A$ be a self-adjoint involutive local $^{*}$-homomorphism. We define anti-linear isometric involution $\gamma^{\phi}:L^{2}(A,\tau)\longrightarrow L^{2}(A,\tau)$ by setting 

\begin{align}\label{EQ.SSEC.NCDS_NCG_QG_10}
\gamma^{\phi}(u):=\phi(u^{*})
\end{align}

\noindent for all $u\in L^{2}(A,\tau)$.

\begin{prp}\label{PRP.Wstar_Derivation_QG_Dynamic}
Let $(A,\tau)$ be a tracial AF-$C^{*}$-algebra. For all generators $\lc\DII,\phi\rc$ of dynamic quantum gradients, we have

\begin{itemize}
\item[1)] AF-$A$-bimodule structure $\lc\phi,\id_{A},\gamma^{\phi}\rc$ on $A$,

\item[2)] quantum gradient $\nabla^{\DII,\phi}:A_{0}\longrightarrow L^{2}(A,\tau)$ defined by

\begin{align}\label{EQ.PRP.Wstar_Derivation_QG_Dynamic_1}
\nabla^{\DII,\phi}x:=L^{-1}\lc\restr{0.925}{\frac{d}{dt}}{t=0,\w}\textrm{Ad}_{t}^{\mathcal{D}}\lc{}L_{\phi(x)}\rc\phi\rc{}=i[\DII,x]_{A}^{\phi}
\end{align}

\begin{reapply}
\end{reapply}

\noindent for all $x\in A_{0}$,

\item[3)] $\nabla^{\DII,\phi,*}:=\lc\nabla^{\DII,\phi}\rc^{*}$ with core $A_{0}$ and determined by

\begin{align}\label{EQ.PRP.Wstar_Derivation_QG_Dynamic_2}
\nabla^{\DII,\phi,*}x=L^{-1}\lc\phi\lc\restr{0.925}{\frac{d}{dt}}{t=0,\w}\textrm{Ad}_{-t}^{\mathcal{D}}\lc L_{x}\phi\rc\rc\phi\rc{}=i\{\DII,x\}_{A}^{\phi}
\end{align}

\begin{reapply}
\end{reapply}

\noindent for all $x\in A_{0}$,

\item[4)] $\Delta^{\DII,\phi}:=\nabla^{\DII,\phi,*}\nabla^{\DII,\phi}$ with core $A_{0}$ and determined by

\begin{align}\label{EQ.PRP.Wstar_Derivation_QG_Dynamic_3}
\Delta^{\DII,\phi}x=-\big\{\DII,[\DII,x]_{A}^{\phi}\big\}_{A}^{\phi}
\end{align}

\begin{reapply}
\end{reapply}

\noindent for all $x\in A_{0}$.
\end{itemize}
\end{prp}
\begin{proof}
We have $1)$ by construction. For both Equation \ref{EQ.PRP.Wstar_Derivation_QG_Dynamic_1} and Equation \ref{EQ.PRP.Wstar_Derivation_QG_Dynamic_2}, we see $1)$ in Definition \ref{DFN.Wstar_Derivation_QG_Dynamic_Generator} shows existence of weak derivatives. Equation \ref{EQ.DFN.Wstar_Derivation_QG_Dynamic_Generator_2} implies the second identity in Equation \ref{EQ.PRP.Wstar_Derivation_QG_Dynamic_1} and therefore Equation \ref{EQ.PRP.Wstar_Derivation_QG_Dynamic_1} itself.\par
Note weak differentiation in Equation \ref{EQ.PRP.Wstar_Derivation_QG_Dynamic_1} is strong differentiation \cite{ART.Chr.2016.Weakly_Differentiable_Operators}. We know all weak, resp.~strong derivatives in use exist. Using the latter and sequential strong continuity of multiplication, we directly verify the Leibniz rule for $\nabla^{\DII,\phi}$. Equation \ref{EQ.PRP.Wstar_Derivation_QG_Dynamic_1} implies symmetry. Thus $\nabla^{\DII,\phi}$ is a symmetric $A_{0}$-derivation. Using self-adjointness of $\phi$ for the second identity below, Equation \ref{EQ.REM.AF_Cstar_Local_Hom_Self_Adjoint_1} and Equation \ref{EQ.DFN.Wstar_Derivation_QG_Dynamic_Generator_1} let us calculate

\begin{align*}
\lgl x,\nabla^{\DII,\phi}y\rgl_{\tau} & = \bigg\langle L^{-1}\lc\restr{0.925}{\frac{d}{dt}}{t=0,\w}\textrm{Ad}_{-t}^{\mathcal{D}}\lc L_{x}\phi\rc\rc{},\phi(y)\bigg\rangle_{\tau} \phantom{\Bigg)} \\
& = \bigg\langle \phi\lc{}L^{-1}\lc\restr{0.925}{\frac{d}{dt}}{t=0,\w}\textrm{Ad}_{-t}^{\mathcal{D}}\lc L_{x}\phi\rc\rc\rc{},y\bigg\rangle_{\tau} \phantom{\Bigg)} \\
& = \bigg\langle L^{-1}\lc\phi\lc\restr{0.925}{\frac{d}{dt}}{t=0,\w}\textrm{Ad}_{-t}^{\mathcal{D}}\lc L_{x}\phi\rc\rc\phi\rc{},y\bigg\rangle_{\tau} \phantom{\Bigg)}
\end{align*}

\noindent for all $t\in\mathbb{R}$ and $x,y\in A_{0}$. Since $A_{0}\subset L^{2}(A,\tau)$ is $\|.\|_{\tau}$-dense, we see the above calculation implies Equation \ref{EQ.PRP.Wstar_Derivation_QG_Dynamic_2}. Hence $A_{0}$ lies in domain of the adjoint. Locality follows by $1.1)$ in Definition \ref{DFN.Wstar_Derivation_QG_Dynamic_Generator}. We have $2)$ and $3)$. They imply $4)$ and therefore Equation \ref{EQ.PRP.Wstar_Derivation_QG_Dynamic_3}.
\end{proof}

\begin{dfn}\label{DFN.Wstar_Derivation_QG_Dynamic}
Let $(A,\tau)$ be a tracial AF-$C^{*}$-algebra. For all generators $\lc\DII,\phi\rc$ of dynamic quantum gradients, we call

\begin{itemize}
\item[1)] $\lc\phi,\id_{A},\gamma^{\phi}\rc$ the $\phi$-intertwined AF-$A$-bimodule structure on $A$,

\item[2)] $\nabla^{\DII,\phi}$ the dynamic quantum gradient generated by $\lc\DII,\phi\rc$,

\item[3)] $\nabla^{\DII,\phi}$ non-twisted if $\phi=\id_{A}$, else twisted.
\end{itemize}
\end{dfn}

\begin{ntn}\label{NTN.Wstar_Derivation_QG_Dynamic}
We suppress $\phi$ in Definition \ref{DFN.Wstar_Derivation_QG_Dynamic_Generator} and Definition \ref{DFN.Wstar_Derivation_QG_Dynamic}, as well as all objects in Proposition \ref{PRP.Wstar_Derivation_QG_Dynamic}, if $\phi=\id_{A}$ or $\mathcal{D}=\mathcal{D}_{\phi}$ as per $2)$ in Definition \ref{DFN.Wstar_Derivation_QG_Intertwining}.
\end{ntn}


\pagebreak


Definition \ref{DFN.Wstar_Derivation_QG_Dynamic_System} gives trace-preserving local $C^{*}$-dynamical systems. Lemma \ref{LEM.Wstar_Derivation_QG_Dynamic_System} yields canonical extensions to conjugation groups. Corollary \ref{COR.Wstar_Derivation_QG_Dynamic_System}, by considering such extensions, gives non-twisted dynamic quantum gradients by norm differentiation of trace-preserving local $C^{*}$-dynamical systems. Note Remark \ref{REM.QOT_Type_II_1}.

\begin{dfn}\label{DFN.Wstar_Derivation_QG_Dynamic_System}
Let $I\in\lset\mathbb{R},[0,\infty)\rset$.

\begin{itemize}
\item[1)] Let $V$ be a Banach space. We say that a semigroup $G:I\longrightarrow\BII(V)$ is strongly continuous if $x=\|.\|_{V}$-$\lim_{t\rightarrow 0}G_{t}(v)$ for all $v\in V$.

\item[2)] Let $A$ be a $C^{*}$-algebra. Let $\mathrm{Aut}(A)\subset\BII(A)$ be the automorphism group of $A$ given by all $^{*}$-isomorphisms. A $C^{*}$-dynamical system $\lc{}A,\mathbb{R},\alpha\rc$ is a strongly continuous group homomorphism $\alpha:\mathbb{R}\longrightarrow\mathrm{Aut}(A)$.

\item[3)] Let $(A,\tau)$ be a tracial AF-$C^{*}$-algebra. We call a $C^{*}$-dynamical system $\lc{}A,\mathbb{R},\alpha\rc$

\begin{itemize}
\item[3.1)] $\tau$-preserving if $\alpha_{t}$ is $\tau$-preserving for all $t\in\mathbb{R}$,

\item[3.2)] local if $\alpha_{t}(A_{j})\subset A_{j}$ for all $t\in\mathbb{R}$ and $j\in\mathbb{N}$.
\end{itemize}

\begin{reapply}
\end{reapply}

\end{itemize}
\end{dfn}

\begin{lem}\label{LEM.Wstar_Derivation_QG_Dynamic_System}
For all $\tau$-preserving local $C^{*}$-dynamical systems $\lc{}A,\mathbb{R},\alpha\rc$, there exists unique $\mathcal{D}_{\alpha}\in\UBII\lc{}L^{2}(A,\tau)\rc_{h}$ s.t.~$L\circ\alpha_{t}=\Ad_{t}^{\mathcal{D}_{\alpha}}\circ L$ for all $t\in\mathbb{R}$.
\end{lem}
\begin{proof}
Let $\lc{}A,\mathbb{R},\alpha\rc$ be a $\tau$-preserving local $C^{*}$-dynamical system. We extend to strongly continuous unitary group on $L^{2}(A,\tau)$ s.t.~Stone's theorem implies our claim. Let $t\in\mathbb{R}$. For all $u,v\in A_{0}$, get $\alpha_{t}(u),\alpha_{t}(v)\in L^{2}(A,\tau)$ by locality, as well as 

\begin{align}\label{EQ.LEM.Wstar_Derivation_QG_Dynamic_System_1}
\lgl \alpha_{t}(u),\alpha_{t}(v)\rgl_{\tau}^{2}=\tau\lc\alpha_{t}(u)^{*}\alpha_{t}(v)\rc{}=\tau\lc\alpha_{t}\lc{}u^{*}v\rc\rc{}=\lgl u,v\rgl_{\tau}^{2}    
\end{align}

\noindent by the $^{*}$-homomorphism property and $\tau$-preservation. Using $A_{0}\subset L^{2}(A,\tau)$ $\|.\|_{\tau}$-dense and Equation \ref{EQ.LEM.Wstar_Derivation_QG_Dynamic_System_1}, we extend to $\alpha_{t}\in\UII\lc\BII\lc{}L^{2}(A,\tau)\rc\rc$ here. We have unitary group $\alpha:\mathbb{R}\longrightarrow\UII\lc\BII\lc{}L^{2}(A,\tau)\rc\rc$. We show its strong continuity. For all $j\in\mathbb{N}$, set $\alpha_{t}^{j}:=\alpha_{t}\vert_{A_{j}}$ for all $t\in\mathbb{R}$. Locality shows we have strongly continuous group $\alpha^{j}:\mathbb{R}\longrightarrow\textrm{Aut}(A_{j})$ w.r.t.~$\|.\|_{A}$ in each case. Finite-dimensionality further implies strong continuity w.r.t.~$\|.\|_{\tau}$.\par
For all $t\in\mathbb{R}$, Equation \ref{EQ.LEM.Wstar_Derivation_QG_Dynamic_System_1} shows $\alpha_{t}\in \UII\lc\BII\lc{}L^{2}(A,\tau)\rc\rc$ is an isometry. Using the latter get uniform bounds, note $3)$ in Proposition \ref{PRP.AF_Cstar_Trace_III} implies $\alpha:\mathbb{R}\longrightarrow\UII\lc\BII\lc{}L^{2}(A,\tau)\rc\rc$ is strongly continuous since $\alpha^{j}$ is for all $j\in\mathbb{N}$. Stone's theorem yields unique generator $\mathcal{D}_{\alpha}\in\UBII\lc{}L^{2}(A,\tau)\rc_{h}$ s.t.~

\begin{align}\label{EQ.LEM.Wstar_Derivation_QG_Dynamic_System_2}
\alpha_{t}=e^{it\mathcal{D}_{\alpha}}
\end{align}

\noindent for all $t\in\mathbb{R}$ \lc{}cf.~Theorem 5.6.36 in \cite{BK.Kad_Rin.1997.OpAlg_I}\rc{}. For all $t\in\mathbb{R}$ and $x\in A$, get $L_{\alpha_{t}(x)}=\alpha_{t}L_{x}\alpha_{t}^{*}$ by the $^{*}$-homomorphism property. Using $A_{0}\subset L^{2}(A,\tau)$ $\|.\|_{\tau}$-dense, Equation \ref{EQ.LEM.Wstar_Derivation_QG_Dynamic_System_2} then implies our claim. Altogether, our proof is extension of invariant $C^{*}$-dynamical systems in our special case \lc{}cf.~Proposition 7.4.12 \cite{BK.Ped.2018.Cstar_Algebras}\rc{}.
\end{proof}

\begin{cor}\label{COR.Wstar_Derivation_QG_Dynamic_System}
Let $(A,\tau)$ be a tracial AF-$C^{*}$-algebra. For all $\tau$-preserving local\linebreak $C^{*}$-dynamical systems $\lc{}A,\mathbb{R},\alpha\rc$, $\lc\mathcal{D}_{\alpha},\id_{A}\rc$ is a generator of dynamic quantum gradient and we have

\begin{itemize}
\item[1)] quantum gradient $\nabla^{\mathcal{D}_{\alpha}}:A_{0}\longrightarrow L^{2}(A,\tau)$ given by

\begin{align}\label{EQ.COR.Wstar_Derivation_QG_Dynamic_System_1}
\nabla^{\mathcal{D}_{\alpha}}x=i[\mathcal{D}_{\alpha},x]_{A}=iL^{-1}\lc\overline{\mathcal{D}_{\alpha}L_{x}-L_{x}\mathcal{D}_{\alpha}}\rc{}
\end{align}

\begin{reapply}
\end{reapply}

\noindent for all $x\in A_{0}$,

\item[2)] $\nabla^{\mathcal{D}_{\alpha},*}=\lc\nabla^{\mathcal{D}_{\alpha}}\rc^{*}$ with core $A_{0}$ and determined by

\begin{align}\label{EQ.COR.Wstar_Derivation_QG_Dynamic_System_2}
\nabla^{\mathcal{D}_{\alpha},*}x=-\nabla^{\mathcal{D}_{\alpha}}x=-i[\mathcal{D}_{\alpha},x]_{A}
\end{align}

\begin{reapply}
\end{reapply}

\noindent for all $x\in A_{0}$,

\item[3)] $\Delta^{\mathcal{D}_{\alpha}}=\nabla^{\mathcal{D}_{\alpha},*}\nabla^{\mathcal{D}_{\alpha}}$ with core $A_{0}$ and determined by

\begin{align}\label{EQ.COR.Wstar_Derivation_QG_Dynamic_System_3}
\Delta^{\mathcal{D}_{\alpha}}x=-\lc\nabla^{\mathcal{D}_{\alpha}}\rc^{2}(x)=[\mathcal{D}_{\alpha},[\mathcal{D}_{\alpha},x]_{A}]_{A}
\end{align}

\begin{reapply}
\end{reapply}

\noindent for all $x\in A_{0}$.
\end{itemize}
\end{cor}
\begin{proof}
Let $j\in\mathbb{N}$ and $x\in A_{j}$. Note we use locality to define strongly continuous group $\alpha^{j}:\mathbb{R}\longrightarrow\textrm{Aut}(A_{j})$ in the proof of Lemma \ref{LEM.Wstar_Derivation_QG_Dynamic_System}. It is local and $\tau$-preserving since we have finite tracial AF-$C^{*}$-algebra $(A_{j},\tau)$ as per Example \ref{BSP.AF_Cstar_Trace_Fin}. Applying Lemma \ref{LEM.Wstar_Derivation_QG_Dynamic_System} to $(A_{j},\tau)$ and $\alpha^{j}$ yields $\mathcal{D}_{j}\in\BII(A_{j})_{h}$ s.t.~

\begin{align}\label{EQ.COR.Wstar_Derivation_QG_Dynamic_System_4}
L_{\alpha_{t}^{j}(x),A_{j}}=\textrm{Ad}_{t}^{\mathcal{D}_{j}}\big(L_{x,A_{j}}\big)
\end{align}

\noindent for all $t\in\mathbb{R}$. The conjugation group $\Ad^{\mathcal{D}_{j}}:\mathbb{R}\longrightarrow\textrm{Aut}\lc\BII(A_{j})\rc$ is norm differentiable at zero for all $S\in\BII(A_{j})$. Thus locality and Equation \ref{EQ.COR.Wstar_Derivation_QG_Dynamic_System_4} imply

\begin{align}\label{EQ.COR.Wstar_Derivation_QG_Dynamic_System_5}
\restr{0.925}{\frac{d}{dt}}{t=0,\|.\|_{A}}\alpha_{t}(x)=\restr{0.925}{\frac{d}{dt}}{t=0,\|.\|_{A}}\alpha_{t}^{j}(x)=iL_{A_{j}}^{-1}\big(\mathcal{D}_{j}L_{x,A_{j}}-L_{x,A_{j}}\mathcal{D}_{j}\big)\in A_{j}.
\end{align}

\noindent Weak, strong and norm differentiation coincide in the finite-dimensional setting. Using normality of canonical left-actions, Lemma \ref{LEM.Wstar_Derivation_QG_Dynamic_System} and Equation \ref{EQ.COR.Wstar_Derivation_QG_Dynamic_System_5} show

\begin{align}\label{EQ.COR.Wstar_Derivation_QG_Dynamic_System_6}
\restr{0.925}{\frac{d}{dt}}{t=0,\w}\textrm{Ad}_{t}^{\mathcal{D}_{\alpha}}(L_{x})=L\lc\restr{0.925}{\frac{d}{dt}}{t=0,\|.\|_{A}}\alpha_{t}(x)\rc\in L(A_{j}).
\end{align}

Since $\alpha_{t}^{*}=\alpha_{-t}$ for all $t\in\mathbb{R}$, Lemma \ref{LEM.Wstar_Derivation_QG_Dynamic_System} and Equation \ref{EQ.COR.Wstar_Derivation_QG_Dynamic_System_6} imply

\begin{align}\label{EQ.COR.Wstar_Derivation_QG_Dynamic_System_7}
\restr{0.925}{\frac{d}{dt}}{t=0,\w}\textrm{Ad}_{-t}^{\mathcal{D}_{\alpha}}(L_{x})=-L\lc\restr{0.925}{\frac{d}{dt}}{t=0,\|.\|_{A}}\alpha_{t}(x)\rc\in L(A_{j}).
\end{align}

Equation \ref{EQ.COR.Wstar_Derivation_QG_Dynamic_System_6} and Equation \ref{EQ.COR.Wstar_Derivation_QG_Dynamic_System_7} imply $1.1)$ in Definition \ref{DFN.Wstar_Derivation_QG_Dynamic_Generator} at once. Using $\alpha_{t}^{*}=\alpha_{-t}$ for all $t\in\mathbb{R}$ and Lemma \ref{LEM.Wstar_Derivation_QG_Dynamic_System}, note Equation \ref{EQ.COR.Wstar_Derivation_QG_Dynamic_System_6} and Equation \ref{EQ.COR.Wstar_Derivation_QG_Dynamic_System_7} show Equation \ref{EQ.DFN.Wstar_Derivation_QG_Dynamic_Generator_1} is given by

\begin{align}\label{EQ.COR.Wstar_Derivation_QG_Dynamic_System_8}
\lgl x,\restr{0.925}{\frac{d}{dt}}{t=0,\|.\|_{A}}\alpha_{-t}(y)\rgl_{\tau}=\lgl\restr{0.925}{\frac{d}{dt}}{t=0,\|.\|_{A}}\alpha_{t}(x),y\rgl_{\tau}
\end{align}

\noindent for all $x,y\in A_{0}$. Equation \ref{EQ.COR.Wstar_Derivation_QG_Dynamic_System_8} shows $1.2)$ in Definition \ref{DFN.Wstar_Derivation_QG_Dynamic_Generator}. We therefore have $1)$ in Definition \ref{DFN.Wstar_Derivation_QG_Dynamic_Generator}, i.e.~$\lc\mathcal{D}_{\alpha},\id_{A}\rc$ is a generator of dynamic quantum gradient. Apply Proposition \ref{PRP.Wstar_Derivation_QG_Dynamic} to $\lc\mathcal{D}_{\alpha},\id_{A}\rc$. Equation \ref{EQ.PRP.Wstar_Derivation_QG_Dynamic_1} shows Equation \ref{EQ.COR.Wstar_Derivation_QG_Dynamic_System_1}. Equation \ref{EQ.COR.Wstar_Derivation_QG_Dynamic_System_1} and Equation \ref{EQ.COR.Wstar_Derivation_QG_Dynamic_System_8} show Equation \ref{EQ.COR.Wstar_Derivation_QG_Dynamic_System_2} and Equation \ref{EQ.COR.Wstar_Derivation_QG_Dynamic_System_3}. Get $1)$ to $3)$.
\end{proof}

Following identities in Lemma \ref{LEM.Wstar_Derivation_QG_Intertwining}, Corollary \ref{COR.Wstar_Derivation_QG_Intertwining_I} gives twisted and non-twisted dynamic quantum gradients by using intertwining self-adjoint unbounded operators as generators of twisted conjugation groups.\par
Definition \ref{DFN.Local_Operator} gives necessary local and strongly local properties underlying both twisted and non-twisted dynamic quantum gradients. Proposition \ref{PRP.Local_Operator} collects several implied properties, in particular splitting of induced semigroups as per Equation \ref{EQ.PRP.Local_Operator_1} applicable to heat semigroups of their quantum Laplacians, used in our discussion.

\begin{dfn}\label{DFN.Local_Operator}
Let $(A,\tau)$ be a tracial AF-$C^{*}$-algebra and $T\in\UBII\lc{}L^{2}(A,\tau)\rc_{h}$.

\begin{itemize}
\item[1)] We say that $T$ is local if $A_{0}\subset\dom T$ and $T(A_{j})\subset A_{j}$ for all $j\in\mathbb{N}$. We say that $T$ is strongly local if for all $j\in\mathbb{N}$ and $x\in A_{j}$, we have

\begin{align}\label{EQ.DFN.Local_Operator_1}
T\pi_{j}^{A}L_{x}\in L(A_{j}),\ T\pi_{j}^{\perp}L_{x}=0.
\end{align}

\noindent 

\item[2)] Let $T$ be local. For all $j\in\mathbb{N}$, set $T_{j}:=\comAj T$ and $T_{j}^{\perp}:=\comAjperp T$.
\end{itemize}
\end{dfn}

\begin{rem}\label{REM.Local_Operator}
Let $(A,\tau)$ be a tracial AF-$C^{*}$-algebra and $T\in\UBII\lc{}L^{2}(A,\tau)\rc_{h}$ strongly local. For all $j\in\mathbb{N}$ and $u\in A_{j}$, $T\pi_{j}^{A}L_{u}\in L(A_{j})$ implies $u\in\dom T$ and

\begin{align}\label{EQ.REM.Local_Operator_1}
T(u)=T\big(\pi_{j}^{A}(u)\big)=T\big(\pi_{j}^{A}(u1_{A_{j}})\big)=\big(T\pi_{j}^{A}L_{u}\big)(1_{A_{j}})\in A_{j}.
\end{align}

\noindent Equation \ref{EQ.REM.Local_Operator_1} shows $T$ is local. Therefore, strongly local implies local.
\end{rem}


\pagebreak


\begin{prp}\label{PRP.Local_Operator}
Let $(A,\tau)$ be a tracial AF-$C^{*}$-algebra and $T\in\UBII\lc{}L^{2}(A,\tau)\rc_{h}$.

\begin{itemize}
\item[1)] $T$ is local if and only if $T:\lc\dom T,\|.\|_{T}\rc\longrightarrow L^{2}(A,\tau)$ has orthonormal eigenbasis $\{e_{n}\}_{n\in\mathbb{N}}\subset A_{0}$ s.t.~it is furthermore orthonormal eigenbasis of $\pi_{j}^{A}$ for all $j\in\mathbb{N}$.

\item[2)] Let $T\in\UBII\lc{}L^{2}(A,\tau)\rc_{h}$ be local. 

\begin{itemize}
\item[2.1)] $T$ has core $A_{0}$. For all $j\in\mathbb{N}$, get $T\in\UBII_{A_{j}}\lc{}L^{2}(A,\tau)\rc$. \phantom{\big)}

\item[2.2)] For all $t\in\mathbb{R}$ and $j\in\mathbb{N}$, we have \phantom{\big)}

\begin{align}\label{EQ.PRP.Local_Operator_1}
e^{itT}=e^{itT_{j}}\oplus e^{itT_{j}^{\perp}}
\end{align}

\begin{reapply}
\end{reapply}

\noindent w.r.t.~$\BII(A_{j})\oplus\BII(A_{j}^{\perp})$.
\end{itemize}

\begin{reapply}
\end{reapply}

\item[3)] Let $T\in\UBII\lc{}L^{2}(A,\tau)\rc_{h}$ be strongly local. For all $j\in\mathbb{N}$, $x\in A_{j}$ and $t\in\mathbb{R}$, we have

\begin{itemize}
\item[3.1)] $TL_{x}=T_{j}L_{x}\in L(A_{j})$ and $L_{x}T=L_{x}T_{j}\in L(A_{j})$, \phantom{\big)}

\item[3.2)] $e^{itT}L_{x}=e^{itT_{j}}L_{x}$ and $L_{x}e^{itT}=L_{x}e^{itT_{j}}$. \phantom{\big)}
\end{itemize}

\begin{reapply}
\end{reapply}

\end{itemize}
\end{prp}
\begin{proof}
We directly verify $1)$. Let $T$ be local. Then $1)$ implies $2.1)$. Using reducibility as per $2.1)$, get $2.2)$ by $2)$ in Corollary \ref{COR.Compression_Preservation_I}. Get $2)$. For all $j\in\mathbb{N}$ and $x\in A_{j}$, $\lb\pi_{j}^{A},L_{x}\rb{}=0$ by $1)$ in Proposition \ref{PRP.AF_Cstar_Bimodule_L2Inf_SR}\rc{}. Let $T$ be strongly local. $T$ is local by Remark \ref{REM.Local_Operator}. We see Equation \ref{EQ.DFN.Local_Operator_1} implies $3.1)$ by $1.3)$ in Proposition \ref{PRP.Reducible} since we have reducibility. Moreover, Equation \ref{EQ.DFN.Local_Operator_1} and Equation \ref{EQ.PRP.Local_Operator_1} show $3.2)$. Get $3)$.
\end{proof}

Strong locality shows Equation \ref{EQ.DFN.Wstar_Derivation_QG_Intertwining_1} is well-defined by $3)$ in Proposition \ref{PRP.Local_Operator}. We use Example \ref{BSP.Wstar_Derivation_QG_Intertwining} for sets of Clifford generators as per Definition \ref{DFN.Wstar_Derivation_QG_Intertwining_Clifford}.

\begin{dfn}\label{DFN.Wstar_Derivation_QG_Intertwining}
Let $(A,\tau)$ be a tracial AF-$C^{*}$-algebra and $\phi:A\longrightarrow A$ a self-adjoint involutive local $^{*}$-homomorphism. Let $\mathcal{D}\in\UBII\lc{}L^{2}(A,\tau)\rc_{h}$.

\begin{itemize}
\item[1)] We say that $\mathcal{D}$ is $\phi$-intertwining if 

\begin{itemize}
\item[1.1)] $\mathcal{D}$ is strongly local,

\item[1.2)] $\phi\lc\dom D\rc\subset\dom\mathcal{D}$, $\mathcal{D}\phi\neq 0$, and $\mathcal{D}\phi=\pm\phi \mathcal{D}$,

\item[1.3)] for all $j\in\mathbb{N}$ and $x,y\in A_{j}$, we have

\begin{align}\label{EQ.DFN.Wstar_Derivation_QG_Intertwining_1}
\lgl x,L^{-1}\lc\mathcal{D}_{j}L_{\phi(y)}-L_{y}\mathcal{D}_{j}\rc\rgl_{\tau}=\lgl L^{-1}\lc\sgn\lc\mathcal{D}\rc{}L_{x}\mathcal{D}_{j}-\mathcal{D}_{j}L_{\phi(x)}\rc{},y\rgl_{\tau}\phantom{\bigg)}.
\end{align}

\begin{reapply}
\end{reapply}

\end{itemize}

\begin{reapply}
\end{reapply}

\item[2)] Let $\mathcal{D}$ be $\phi$-intertwining. Let $\sgn\lc\mathcal{D}\rc\in\lset\hspace{-0.025cm} \pm 1\rset$ s.t.~$\mathcal{D}\phi=\sgn\lc\mathcal{D}\rc\phi\mathcal{D}$ be its sign. Its sign delta is $\delta\lc\DII\rc{}:=\delta_{-1}\lc\sgn\lc\DII\rc\rc\in\lset{}0,1\rset$. Set

\begin{align}
\mathcal{D}_{\phi}:=(-i)^{\delta\lc\DII\rc{}}\DII\phi.
\end{align}

\begin{reapply}
\end{reapply}

\end{itemize}
\end{dfn}

\begin{bsp}\label{BSP.Wstar_Derivation_QG_Intertwining}
Let $(A,\tau)$ be a tracial AF-$C^{*}$-algebra and $\phi:A\longrightarrow A$ a self-adjoint involutive local $^{*}$-homomorphism. Let $d\in L^{\infty}(A,\tau)_{h}\setminus\lset{}0\rset$ s.t.~$L_{d}$ is strongly local and $\phi\lc{}d\rc{}=-d$. We show $L_{d}$ is $\phi$-intertwining s.t.~$\sgn\lc{}L_{d}\rc{}=-1$.\par
We know $1.1)$ in Definition \ref{DFN.Wstar_Derivation_QG_Intertwining}. Equation \ref{EQ.REM.AF_Cstar_Local_Hom_Self_Adjoint_1} implies $-L_{d}=L_{\phi\lc{}d\rc{}}=\phi L_{d}\phi$ and therefore $1.2)$ in Definition \ref{DFN.Wstar_Derivation_QG_Intertwining}. We moreover have $\sgn\lc{}L_{d}\rc{}=-1$. Compressing as per Corollary \ref{COR.AF_Cstar_Bimodule_Compression_Restriction} and using $1)$ in Proposition \ref{PRP.AF_Cstar_Bimodule_L2Inf_SR}, we calculate

\begin{align*}
\lgl x,L^{-1}\big(\pi_{j}^{A}L_{d}\pi_{j}^{A}L_{\phi(y)}-L_{y}\pi_{j}^{A}L_{d}\pi_{j}^{A}\big)\rgl_{\tau} & = \lgl x,\big(\restr{0.875}{L}{A_{j}}\big)^{-1}\big(\pi_{j}^{A}(L_{d\phi(y)-yd})\pi_{j}^{A}\big)\rgl_{\tau} \phantom{\bigg)} \\
& = \lgl L^{-1}\big(-L_{x}\pi_{j}^{A}L_{d}\pi_{j}^{A}-\pi_{j}^{A}L_{d}\pi_{j}^{A}L_{\phi(x)}\big),y\rgl_{\tau} \phantom{\bigg)}
\end{align*}

\noindent for all $j\in\mathbb{N}$ and $x,y\in A_{j}$. The above calculation shows Equation \ref{EQ.DFN.Wstar_Derivation_QG_Intertwining_1} in our case since $\sgn\lc{}L_{d}\rc{}=-1$. Thus $1.3)$ in Definition \ref{DFN.Wstar_Derivation_QG_Intertwining}, hence our claim holds.
\end{bsp}

\begin{lem}\label{LEM.Wstar_Derivation_QG_Intertwining}
Let $(A,\tau)$ be a tracial AF-$C^{*}$-algebra and $\phi:A\longrightarrow A$ a self-adjoint involutive local $^{*}$-homomorphism. For all $\phi$-intertwining $\mathcal{D}\in\UBII\lc{}L^{2}(A,\tau)\rc_{h}$, $\lc\mathcal{D}_{\phi},\phi\rc$ is a generator of dynamic quantum gradient and we have

\begin{itemize}
\item[1)] $[\mathcal{D}_{\phi},x]_{A}^{\phi}=(-i)^{\delta\lc\DII\rc{}}L^{-1}\lc\overline{\mathcal{D}L_{x}-L_{\phi(x)}\mathcal{D}}\rc$,

\item[2)] $\{\mathcal{D}_{\phi},x\}_{A}^{\phi}=-(-i)^{\delta\lc\DII\rc{}}L^{-1}\lc\overline{\sgn\lc\mathcal{D}\rc\mathcal{D}L_{x}-L_{\phi(x)}\mathcal{D}}\rc$,

\item[3)] $\big\{\mathcal{D}_{\phi},[\mathcal{D}_{\phi},x]_{A}^{\phi}\big\}_{A}^{\phi}=-L^{-1}\lc\overline{\mathcal{D}^{2}L_{x}+L_{x}\mathcal{D}^{2}-2\mathcal{D}L_{\phi(x)}\mathcal{D}}\rc$,
\end{itemize}

\noindent for all $x\in A_{0}$.
\end{lem}
\begin{proof}
Let $\mathcal{D}\in\UBII\lc{}L^{2}(A,\tau)\rc_{h}$ be $\phi$-intertwining. Then $1.1)$ and $1.2)$ in Definition \ref{DFN.Wstar_Derivation_QG_Intertwining} imply $\mathcal{D}_{\phi}$ is strongly local since $\lb\pi_{j}^{A},\phi\rb{}=0$ for all $j\in\mathbb{N}$ by $3.1)$ in Proposition \ref{PRP.AF_Cstar_Local_Hom_I}. In addition, $\phi\lc\dom \mathcal{D}_{\phi}\rc\subset\dom \mathcal{D}_{\phi}$ and $\mathcal{D}_{\phi}\phi=\sgn\lc\mathcal{D}\rc\phi \mathcal{D}_{\phi}$.\par
Let $j\in\mathbb{N}$ and $x\in A_{j}$. Note $\lb\pi_{j}^{A},\phi\rb{}=0$. Set

\begin{align}\label{EQ.LEM.Wstar_Derivation_QG_Intertwining_1}
\mathcal{D}_{\phi,j}:=\comAj\mathcal{D}_{\phi}=(-i)^{\delta\lc\DII\rc{}}\mathcal{D}_{j}\phi,\ \mathcal{D}_{\phi,j}^{\perp}:=\comAjperp\mathcal{D}_{\phi}=(-i)^{\delta\lc\DII\rc{}}\mathcal{D}_{j}^{\perp}\phi.
\end{align}

\noindent The bounded operators in Equation \ref{EQ.LEM.Wstar_Derivation_QG_Intertwining_1} are those in Proposition \ref{PRP.Local_Operator} for $\mathcal{D}_{\phi}$. Using $3.2)$ in Proposition \ref{PRP.Local_Operator} and $L_{\phi(x)}=\phi L_{x}\phi$, the first identity in Equation \ref{EQ.LEM.Wstar_Derivation_QG_Intertwining_1} shows

\begin{align}\label{EQ.LEM.Wstar_Derivation_QG_Intertwining_2}
\textrm{Ad}_{t}^{\mathcal{D}_{\phi}}(L_{x})=\textrm{Ad}_{t}^{\mathcal{D}_{\phi,j}}(L_{x})
\end{align}

\noindent and

\begin{align}\label{EQ.LEM.Wstar_Derivation_QG_Intertwining_3}
\textrm{Ad}_{-t}^{\mathcal{D}_{\phi}}\lc L_{x}\phi\rc=\textrm{Ad}_{-t}^{\mathcal{D}_{\phi,j}}\lc L_{x}\phi\rc
\end{align}

\noindent for all $t\in\mathbb{R}$.\par


\pagebreak


Equation \ref{EQ.LEM.Wstar_Derivation_QG_Intertwining_2} shows

\begin{align}\label{EQ.LEM.Wstar_Derivation_QG_Intertwining_4}
\restr{0.925}{\frac{d}{dt}}{t=0,\w}\textrm{Ad}_{t}^{\mathcal{D}_{\phi}}(L_{x})=\overline{i\lc\mathcal{D}_{\phi}L_{x}-L_{x}\mathcal{D}_{\phi}\rc{}}=i\lc\mathcal{D}_{\phi,j}L_{x}-L_{x}\mathcal{D}_{\phi,j}\rc{},
\end{align}

\noindent whereas Equation \ref{EQ.LEM.Wstar_Derivation_QG_Intertwining_3} shows

\begin{align}\label{EQ.LEM.Wstar_Derivation_QG_Intertwining_5}
\restr{0.925}{\frac{d}{dt}}{t=0,\w}\textrm{Ad}_{-t}^{\mathcal{D}_{\phi}}\lc L_{x}\phi\rc=\overline{i\lc{}L_{x}\phi \mathcal{D}_{\phi}-\mathcal{D}_{\phi}L_{x}\phi\rc{}}=i\lc{}L_{x}\phi\mathcal{D}_{\phi,j}-\mathcal{D}_{\phi,j}L_{x}\phi\rc{}.
\end{align}

\noindent Using $3.1)$ in Proposition \ref{PRP.Local_Operator} and $L_{\phi(x)}=\phi L_{x}\phi$, the first identity in Equation \ref{EQ.LEM.Wstar_Derivation_QG_Intertwining_1} further shows

\begin{align}\label{EQ.LEM.Wstar_Derivation_QG_Intertwining_6}
\mathcal{D}_{\phi,j}L_{x}-L_{x}\mathcal{D}_{\phi,j}=(-i)^{\delta\lc\DII\rc{}}\sgn\lc\mathcal{D}\rc\phi\lc\mathcal{D}_{j}L_{x}-L_{\phi(x)}\mathcal{D}_{j}\rc\in\phi L(A_{j})
\end{align}

\noindent and 

\begin{align}\label{EQ.LEM.Wstar_Derivation_QG_Intertwining_7}
L_{x}\phi \mathcal{D}_{\phi,j}-\mathcal{D}_{\phi,j}L_{x}\phi=(-i)^{\delta\lc\DII\rc{}}\lc\sgn\lc\mathcal{D}\rc{}L_{x}\mathcal{D}_{j}-\mathcal{D}_{j}L_{\phi(x)}\rc\in L(A_{j}).
\end{align}

Equation \ref{EQ.LEM.Wstar_Derivation_QG_Intertwining_4} and Equation \ref{EQ.LEM.Wstar_Derivation_QG_Intertwining_6} in turn show

\begin{align}\label{EQ.LEM.Wstar_Derivation_QG_Intertwining_8}
\restr{0.925}{\frac{d}{dt}}{t=0,\w}\textrm{Ad}_{t}^{\mathcal{D}_{\phi}}(L_{x})=(-i)^{\delta\lc\DII\rc{}}\sgn\lc\mathcal{D}\rc\phi\lc\mathcal{D}_{j}L_{x}-L_{\phi(x)}\mathcal{D}_{j}\rc\in\phi L(A_{j}),
\end{align}

\noindent whereas Equation \ref{EQ.LEM.Wstar_Derivation_QG_Intertwining_5} and Equation \ref{EQ.LEM.Wstar_Derivation_QG_Intertwining_7} show

\begin{align}\label{EQ.LEM.Wstar_Derivation_QG_Intertwining_9}
\restr{0.925}{\frac{d}{dt}}{t=0,\w}\textrm{Ad}_{-t}^{\mathcal{D}_{\phi}}\lc L_{x}\phi\rc=(-i)^{\delta\lc\DII\rc{}}\lc\sgn\lc\mathcal{D}\rc{}L_{x}\mathcal{D}_{j}-\mathcal{D}_{j}L_{\phi(x)}\rc\in L(A_{j}).
\end{align}

Equation \ref{EQ.LEM.Wstar_Derivation_QG_Intertwining_8} and Equation \ref{EQ.LEM.Wstar_Derivation_QG_Intertwining_9} imply $1.1)$ in Definition \ref{DFN.Wstar_Derivation_QG_Dynamic_Generator} at once. Using Equation \ref{EQ.LEM.Wstar_Derivation_QG_Intertwining_8} for the first, Equation \ref{EQ.DFN.Wstar_Derivation_QG_Intertwining_1} for the second, and finally Equation \ref{EQ.LEM.Wstar_Derivation_QG_Intertwining_9} for the third identity below, we calculate

\begin{align*}
\bigg\langle x,L^{-1}\lc\restr{0.925}{\frac{d}{dt}}{t=0,\w}\textrm{Ad}_{t}^{\mathcal{D}_{\phi}}(L_{y})\phi\rc\bigg\rangle_{\tau}\phantom{\bigg)} =& \ (-i)^{\delta\lc\DII\rc{}}\lgl x,L^{-1}\lc\sgn\lc\mathcal{D}\rc\phi\lc\mathcal{D}_{j}L_{y}-L_{\phi(y)}\mathcal{D}_{j}\rc\phi\rc\rgl_{\tau} \phantom{\Bigg)} \\
=& \ (-i)^{\delta\lc\DII\rc{}}\lgl L^{-1}\lc\sgn\lc\mathcal{D}\rc{}L_{x}\mathcal{D}_{j}-\mathcal{D}_{j}L_{\phi(x)}\rc{},y\rgl_{\tau} \phantom{\Bigg)} \\
=& \ \bigg\langle L^{-1}\lc\restr{0.925}{\frac{d}{dt}}{t=0,\w}\textrm{Ad}_{-t}^{\mathcal{D}_{\phi}}\lc L_{x}\phi\rc\rc{},y\bigg\rangle_{\tau} \phantom{\Bigg)}
\end{align*}

\noindent for all $j\in\mathbb{N}$ and $x,y\in A_{j}$. The above calculation shows $1.2)$ in Definition \ref{DFN.Wstar_Derivation_QG_Dynamic_Generator}. We therefore have $1)$ in Definition \ref{DFN.Wstar_Derivation_QG_Dynamic_Generator}, i.e.~$\lc\mathcal{D}_{\phi},\phi\rc$ is a generator of dynamic quantum gradient. Apply Proposition \ref{PRP.Wstar_Derivation_QG_Dynamic} to $\lc\mathcal{D}_{\phi},\phi\rc$.\par


\pagebreak


Using Equation \ref{EQ.DFN.Wstar_Derivation_QG_Dynamic_Generator_2}, we directly verify $1)$ and $2)$. We show $3)$. Note $A_{0}\subset\dom\mathcal{D}^{2}$ by locality. For all $x,u\in A_{0}$, we have $xu\in\dom\mathcal{D}^{2}$. Using Equation \ref{EQ.DFN.Wstar_Derivation_QG_Dynamic_Generator_2}, strong locality and $(-i)^{2\delta\lc\DII_{\phi}\rc{}}=\sgn\lc\DII\rc$, we apply $1)$ and $2)$ in each finite-dimensional case to get

\begin{align}\label{EQ.LEM.Wstar_Derivation_QG_Intertwining_10}
L\lc\big\{\mathcal{D}_{\phi},[\mathcal{D}_{\phi},x]_{A}^{\phi}\big\}_{A}^{\phi}\rc{}(u)=-\lc\mathcal{D}_{j}^{2}L_{x}+L_{x}\mathcal{D}_{j}^{2}-2\mathcal{D}_{j}L_{\phi(x)}\mathcal{D}_{j}\rc{}(u)
\end{align}

\noindent for all $j\in\mathbb{N}$, $x\in A_{j}$ and $u\in A_{0}$. Using strong locality, Equation \ref{EQ.LEM.Wstar_Derivation_QG_Intertwining_10} shows

\begin{align}\label{EQ.LEM.Wstar_Derivation_QG_Intertwining_11}
L\lc\big\{\mathcal{D}_{\phi},[\mathcal{D}_{\phi},x]_{A}^{\phi}\big\}_{A}^{\phi}\rc{}(u)=-\lc\pi_{k}^{A}\lc\mathcal{D}^{2}L_{x}+L_{x}\mathcal{D}^{2}-2\mathcal{D}L_{\phi(x)}\mathcal{D}\rc\pi_{k}^{A}\rc{}(u)
\end{align}

\noindent for all $j\leq k$ in $\mathbb{N}$, $x\in A_{j}$ and $u\in A_{0}$. For fix but arbitrary $u\in A_{0}$, $\pi_{k}^{A}$ on the right-hand side of the inner bracket in Equation \ref{EQ.LEM.Wstar_Derivation_QG_Intertwining_11} vanishes without loss of generality. Using $3)$ in Proposition \ref{PRP.AF_Cstar_Trace_III}, letting $k\uparrow\infty$ in Equation \ref{EQ.LEM.Wstar_Derivation_QG_Intertwining_11} yields

\begin{align}\label{EQ.LEM.Wstar_Derivation_QG_Intertwining_12}
L\lc\big\{\mathcal{D}_{\phi},[\mathcal{D}_{\phi},x]_{A}^{\phi}\big\}_{A}^{\phi}\rc{}(u)=-\lc\mathcal{D}^{2}L_{x}+L_{x}\mathcal{D}^{2}-2\mathcal{D}L_{\phi(x)}\mathcal{D}\rc{}(u)
\end{align}

\noindent for all $x\in A_{j}$ and $u\in A_{0}$. Note the left-hand side of Equation \ref{EQ.LEM.Wstar_Derivation_QG_Intertwining_12} evaluates a bounded operator. Since $A_{0}\subset L^{2}(A,\tau)$ is $\|.\|_{\tau}$-dense, the right-hand side of Equation \ref{EQ.LEM.Wstar_Derivation_QG_Intertwining_11} in fact evaluates a bounded and closable operator defined on $\dom\mathcal{D}^{2}$. Get $3)$ by closure.
\end{proof}

\begin{cor}\label{COR.Wstar_Derivation_QG_Intertwining_I}
Let $(A,\tau)$ be a tracial AF-$C^{*}$-algebra and $\phi:A\longrightarrow A$ a self-adjoint involutive local $^{*}$-homomorphism. For all $\phi$-intertwining $\mathcal{D}\in\UBII\lc{}L^{2}(A,\tau)\rc_{h}$, $\lc\mathcal{D}_{\phi},\phi\rc$ is a generator of dynamic quantum gradient and we have

\begin{itemize}
\item[1)] quantum gradient $\nabla^{\mathcal{D}_{\phi}}:A_{0}\longrightarrow L^{2}(A,\tau)$ given by

\begin{align}\label{EQ.COR.Wstar_Derivation_QG_Intertwining_I_1}
\nabla^{\mathcal{D}_{\phi}}x=i[\mathcal{D}_{\phi},x]_{A}^{\phi}=i(-i)^{\delta\lc\DII\rc{}}L^{-1}\lc\overline{\mathcal{D}L_{x}-L_{\phi(x)}\mathcal{D}}\rc{}
\end{align}

\begin{reapply}
\end{reapply}

\noindent for all $x\in A_{0}$,

\item[2)] $\nabla^{\mathcal{D}_{\phi},*}=\lc\nabla^{\mathcal{D}_{\phi}}\rc^{*}$ with core $A_{0}$ and determined by

\begin{align}\label{EQ.COR.Wstar_Derivation_QG_Intertwining_I_2}
\nabla^{\mathcal{D}_{\phi},*}x=i\{\mathcal{D}_{\phi},x\}_{A}^{\phi}=-i(-i)^{\delta\lc\DII\rc{}}L^{-1}\lc\overline{\sgn\lc\mathcal{D}\rc\mathcal{D}L_{x}-L_{\phi(x)}\mathcal{D}}\rc{}
\end{align}

\begin{reapply}
\end{reapply}

\noindent for all $x\in A_{0}$,

\item[3)] $\Delta^{\mathcal{D}_{\phi}}=\nabla^{\mathcal{D}_{\phi},*}\nabla^{\mathcal{D}_{\phi}}$ with core $A_{0}$ and determined by

\begin{align}\label{EQ.COR.Wstar_Derivation_QG_Intertwining_I_3}
\Delta^{\mathcal{D}_{\phi}}x=-\big\{\mathcal{D}_{\phi},[\mathcal{D}_{\phi},x]_{A}^{\phi}\big\}_{A}^{\phi}=L^{-1}\lc\overline{\mathcal{D}^{2}L_{x}+L_{x}\mathcal{D}^{2}-2\mathcal{D}L_{\phi(x)}\mathcal{D}}\rc{}
\end{align}

\begin{reapply}
\end{reapply}

\noindent for all $x\in A_{0}$.
\end{itemize} 
\end{cor}
\begin{proof}
Apply Proposition \ref{PRP.Wstar_Derivation_QG_Dynamic} and Lemma \ref{LEM.Wstar_Derivation_QG_Intertwining}.
\end{proof}


\pagebreak


\begin{cor}\label{COR.Wstar_Derivation_QG_Intertwining_II}
Let $(A,\tau)$ be a tracial AF-$C^{*}$-algebra and $\phi:A\longrightarrow A$ a self-adjoint involutive local $^{*}$-homomorphism. Let $\mathcal{D}\in\UBII\lc{}L^{2}(A,\tau)\rc_{h}$ be $\phi$-intertwining. For all $j\in\mathbb{N}$ and $x\in A_{j}$, we have

\begin{itemize}
\item[1)] $\nabla^{\mathcal{D}_{\phi}}x=i(-i)^{\delta\lc\DII\rc{}}L^{-1}\lc\mathcal{D}_{j}L_{x}-L_{\phi(x)}\mathcal{D}_{j}\rc$,

\item[2)] $\nabla^{\mathcal{D}_{\phi},*}x=-i(-i)^{\delta\lc\DII\rc{}}L^{-1}\lc\sgn\lc\mathcal{D}\rc\mathcal{D}_{j}L_{x}-L_{\phi(x)}\mathcal{D}_{j}\rc$,

\item[3)] $\Delta^{\mathcal{D}_{\phi}}x=L^{-1}\lc\mathcal{D}_{j}^{2}L_{x}+L_{x}\mathcal{D}_{j}^{2}-2\mathcal{D}_{j}L_{\phi(x)}\mathcal{D}_{j}\rc$.
\end{itemize}
\end{cor}
\begin{proof}
Apply Corollary \ref{COR.Wstar_Derivation_QG_Intertwining_I} in each finite-dimensional case.
\end{proof}

Definition \ref{DFN.Wstar_Derivation_QG_Intertwining_Clifford} gives intertwining sets of Clifford generators. In the logarithmic mean setting, Example \ref{BSP.L2W_Ric_Wstar_Derivation_QG_Intertwining_Clifford} shows their direct sum quantum gradients yield strictly positive lower Ricci bounds. This requires Lemma \ref{LEM.Wstar_Derivation_QG_Intertwining_Clifford}.

\begin{dfn}\label{DFN.Wstar_Derivation_QG_Intertwining_Clifford}
Let $(A,\tau)$ be a tracial AF-$C^{*}$-algebra and $\phi:A\longrightarrow A$ a self-adjoint involutive local $^{*}$-homomorphism. Let $m\in\mathbb{N}$ and $\lset{}d_{n}\rset_{n=1}^{m}\subset L^{\infty}(A,\tau)_{h}$.

\begin{itemize}
\item[1)] We say that $\lset{}d_{n}\rset_{n=1}^{m}$ is a $\phi$-intertwining set of Clifford generators for $C>0$ if for all $n,k\in\lset{}1,\ldots,m\rset$, we have

\begin{itemize}
\item[1.1)] $L_{d_{n}}$ strongly local and $\phi\lc{}d_{n}\rc{}=-d_{n}$,

\item[1.2)] $d_{n}d_{k}+d_{k}d_{n}=2C\delta_{nk}1_{A}$.
\end{itemize}

\begin{reapply}
\end{reapply}

\item[2)] Let $\lset{}d_{n}\rset_{n=1}^{m}$ be a $\phi$-intertwining set of Clifford generators for $C>0$ as above. For all $n\in\lset{}1,\ldots,m\rset$, set $\partial_{n}:=\nabla^{-iL_{d_{n}},\phi}$ and $\Delta_{n}:=\Delta^{-iL_{d_{n}},\phi}$.
\end{itemize}
\end{dfn}

\begin{lem}\label{LEM.Wstar_Derivation_QG_Intertwining_Clifford}
Let $(A,\tau)$ be a tracial AF-$C^{*}$-algebra, $\phi:A\longrightarrow A$ a self-adjoint involutive local $^{*}$-homomorphism and $m\in\mathbb{N}$. If $\lset{}d_{n}\rset_{n=1}^{m}\subset L^{\infty}(A,\tau)_{h}$ is a $\phi$-intertwining set of Clifford generators for $C>0$, then

\begin{align}\label{EQ.LEM.Wstar_Derivation_QG_Intertwining_Clifford_1}
\partial_{n}\Delta_{k}=\big(\Delta_{k}+\delta_{nk}4C\cdot I\big)\partial_{n}
\end{align}

\noindent for all $n,k\in\lset{}1,\ldots,m\rset$.
\end{lem}
\begin{proof}
Let $\lset{}d_{n}\rset_{n=1}^{m}\subset L^{\infty}(A,\tau)_{h}$ be a $\phi$-intertwining set of Clifford generators for $C>0$. Lemma \ref{LEM.Wstar_Derivation_QG_Intertwining_Clifford_Identities} gives three identities we use in this proof. If $n,k\in\lset{}1,\ldots,m\rset$ s.t.~$n\neq k$, then we see Equation \ref{EQ.LEM.Wstar_Derivation_QG_Intertwining_Clifford_Identities_1} and Equation \ref{EQ.LEM.Wstar_Derivation_QG_Intertwining_Clifford_Identities_2} let us calculate

\begin{align}\label{EQ.LEM.Wstar_Derivation_QG_Intertwining_Clifford_2}
\partial_{n}\Delta_{k}=\partial_{n}\mathrlap{\phantom{\partial}^{*}}\partial_{k}\partial_{k}=(-1)^{2}\cdot \Delta_{k}\partial_{n}=\Delta_{k}\partial_{n}.
\end{align}


\pagebreak


Equation \ref{EQ.LEM.Wstar_Derivation_QG_Intertwining_Clifford_Identities_1} implies $\partial_{n}^{2}=0$ in each case. If $n=k$, then we see Equation \ref{EQ.LEM.Wstar_Derivation_QG_Intertwining_Clifford_Identities_3} together with $\partial_{n}^{2}=0$ lets us calculate

\begin{align}\label{EQ.LEM.Wstar_Derivation_QG_Intertwining_Clifford_3}
\partial_{n}\Delta_{n}=4C\partial_{n}=\big(\Delta_{n}+4C\cdot I\big)\partial_{n}.
\end{align}

\noindent Equation \ref{EQ.LEM.Wstar_Derivation_QG_Intertwining_Clifford_2} and Equation \ref{EQ.LEM.Wstar_Derivation_QG_Intertwining_Clifford_3} show Equation \ref{EQ.LEM.Wstar_Derivation_QG_Intertwining_Clifford_1}.
\end{proof}


\subsection{Noncommutative differential structures and compatibility}\label{SSEC.NCDS_NCG_Notion}

Noncommutative differential structures collect the data which define quantum optimal transport distances. Each consists of two components and one setting. The data collected is compatible with compression and finite-dimensional approximation. These are two general operations we formalise in a coarse graining process.


\subsubsection*{The notion of noncommutative differential structure}

This chapter provides all necessary data. Definition \ref{DFN.NCDS} gives noncommutative differential structures. We explain our notions of compression and finite-dimensional approximation, as well as compatibility with either. To this end, we use the terms \textit{noncommutative} and \textit{quantum} in our discussion as means to distinguish classes of objects as per Figure 2.1.\par
We further explain the data for $1)$ in Definition \ref{DFN.NCDS} satisfies such compatibility by construction. In Subsection \ref{SSEC.QOT_DT_CEN} and Subsection \ref{SSEC.QOT_DT_MG}, we show compatibility transfers to quantum optimal transport. In Subsection \ref{SSEC.QOT_CG}, we then formalise compatibility in the coarse graining process as per Diagram \ref{EQ.SSEC.QOT_QIT_Encoding_16}. This completes our explanation.

\begin{dfn}\label{DFN.NCDS}
Let $(A,\tau)$ and $(B,\omega)$ be tracial AF-$C^{*}$-algebras. Let $(\phi,\bpsi,\gamma)$ be an AF-$A$-bimodule structure on $B$. Let $f$ be symmetric representing function of an operator mean and $\theta\in [0,1]$ s.t.~$\|\omega\|^{1-\theta}<\infty$. Let $\nabla:A_{0}\longrightarrow L^{2}(B,\omega)$ be a quantum gradient.

\begin{itemize}
\item[1)] We call $(\phi,\bpsi,\gamma,\nabla)$ noncommutative differential structure for $(A,\tau)$ and $(B,\omega)$ in $\lc{}f,\theta\rc$-setting.

\item[2)] For all $j\in\mathbb{N}$, we consider the induced AF-$A_{j}$-bimodule structure $(\phi_{j},\bpsi_{j},\gamma_{j})$ on $B_{j}$ as per $4)$ in Definition \ref{DFN.AF_Cstar_Bimodule} together with the $j$-th restricted quantum gradient $\nabla_{\hspace{-0.055cm} j}:A_{j}\longrightarrow B_{j}$ as per $2)$ in Definition \ref{DFN.Wstar_Derivation_QG} and call $\lc\phi_{j},\bpsi_{j},\gamma_{j},\nabla_{\hspace{-0.055cm} j}\rc$ the induced noncommutative differential structure for $(A_{j},\tau)$ and $(B_{j},\omega)$ in $\lc{}f,\theta\rc$-setting.
\end{itemize}
\end{dfn}

\begin{rem}\label{REM.NCDS}
Definition \ref{DFN.NCDS} is motivated by Definition 4.7 in \cite{ART.Car_Maa.2020.Quantum_OT_III}. The latter uses absolutely continuous finite weights \cite{BK.Tak.2003.OpAlg_II} w.r.t.~a given finite trace. Proposition 4.12 in \cite{ART.Car_Maa.2020.Quantum_OT_III} shows a detailed balance condition for Laplacians. We see \cite{ART.Car_Maa.2020.Quantum_OT_III} generalises \cite{ART.Maa.2011.Discrete_OT_Markov}. Yet the detailed balance condition as per Proposition 4.12 in \cite{ART.Car_Maa.2020.Quantum_OT_III} implies ergodicity of the given noncommutative heat semigroup. As such, Definition 4.7 in \cite{ART.Car_Maa.2020.Quantum_OT_III} assumes the ergodic finite-dimensional setting but not traciality, whereas Definition \ref{DFN.NCDS} assumes it but allows for infinite dimensions, possibly non-finite traces, as well as non-ergodicity of noncommutative heat semigroups. We account for these differences.
\end{rem}

\begin{figure}\label{FIG.NC_Setting_Decomposition}
    \begin{center}
    {
        \offinterlineskip

        \hspace*{2.15cm}\vspace*{0.0795cm}\MyHBox[\dimexpr3.4cm+6\fboxsep\relax]{}\par

        \hspace*{2.15cm}\vspace*{0.0795cm}\MyHBox{Commutative}\MyHBox{Not Commutative}\par

        \MyTBox{\hspace{-1.625cm}{Quantum}}{$C_{0}\lc\mathbb{N})$}{$\KII\lc\ell^{2}\lc\mathbb{N}))$}

        \MyTBox{\hspace{-1.625cm}{Not Quantum}}{$C_{0}\lc\mathbb{R})$}{$C_{0}\lc\mathbb{R})\otimes\KII\lc\ell^{2}\lc\mathbb{N}))$}

    }
    \end{center}
    \caption{Matrix for example $C^{*}$-algebras decomposing the noncommutative setting according to commutativity and inclusion in the AF-$C^{*}$-setting. The noncommutative setting subsumes the commutative and properly noncommutative one. Note all function spaces use elements evaluating in complex numbers and vanishing at infinity.}
\end{figure}

We use the two terms \textit{noncommutative} and \textit{quantum} in our discussion as means to distinguish classes of objects as per Figure 2.1. The former denotes objects in the full noncommutative setting, in particular the AF-$C^{*}$-setting. The latter denotes objects in the AF-$C^{*}$-setting compatible with compression and finite-dimensional approximation. Note tracial AF-$C^{*}$-algebras generating hyperfinite factors of type I and II by $\sigma$-weak closure, i.e.~Example \ref{BSP.QOT_Type_I}, Example \ref{BSP.QOT_Type_II_1} and Example \ref{BSP.QOT_Type_II_Infty}, are common algebras of observables in quantum statistical mechanics \cite{BK.Bra.1987.OpAlg_Quantum_StM_I}\cite{BK.Bra.1987.OpAlg_Quantum_StM_II}\cite{BK.Nes_Sto.2006.Rel_Ent}.\par
We use the above to explain compression and finite-dimensional approximation, as well as compatibility with either. For compression, we apply compression maps to tracial AF-$C^{*}$-algebras as per Remark \ref{REM.Cstar_Trace_Abstract_Dualisation}. It acts on and yields objects and properties in the noncommutative setting. For finite-dimensional approximation, we apply restriction maps, possibly up to rescaling as per $1)$ in Definition \ref{DFN.AF_Cstar_Trace_Dualisation_Paths}, to tracial AF-$C^{*}$-algebras as per Definition \ref{DFN.AF_Cstar_Trace_Dualisation_Restriction}. It acts on objects and properties in the AF-$C^{*}$-setting and yields description of these as limits of restricted analogues in the finite-dimensional setting. If we have notions of compression and finite-dimensional approximation for a class of objects or properties, which we give explicitly for each use case in our discussion, then we say such a class is compatible with both. We use compression and finite-dimensional approximation for the coarse graining process as per Diagram \ref{EQ.SSEC.QOT_QIT_Encoding_16}. This demands data compatible with both. The data for $1)$ in Definition \ref{DFN.NCDS} satisfies such compatibility by their locality properties. The coarse graining process, hence compatibility, is essential for our discussion because it reduces the AF-$C^{*}$-setting to the finite-dimensional one s.t.~ergodicity is recovered up to a controlled remainder. Note Remark \ref{REM.NCDS}.


\chapter{Quantum Optimal Transport}\label{CH.QOT}

Quantum optimal transport is described using dynamic transport distances of states on tracial AF-$C^{*}$-algebras. Noncommutative differential structures collect the data which define such dynamic transport distances. First, quantum gradients define continuity equations for states on tracial AF-$C^{*}$-algebras. Continuity equations in turn define sets of admissible paths. Secondly, quasi-entropies define energy functionals by integrating their own evaluation on admissible paths. Minimising square roots of energy functionals over all admissible paths for fixed marginals defines dynamic transport distances, called quantum optimal transport distances. This follows the classical case \cite{ART.Dol_Naz_Sav.2009.Generalised_OT}. We show our construction extends the discrete cases \cite{ART.Maa.2011.Discrete_OT_Markov}\cite{ART.Mie.2011.Discrete_OT_RctDiff}, as well as tracial finite-dimensional ones in \cite{ART.Car_Maa.2014.Quantum_OT_I}\cite{ART.Car_Maa.2017.Quantum_OT_II}\cite{ART.Car_Maa.2020.Quantum_OT_III}. We provide fundamental example classes. The latter themselves yield quantum optimal transport of normal states on hyperfinite factors of type I and II \cite{BK.Ped.2018.Cstar_Algebras}. An application is given by first and second quantisation of spectral triples \cite{ART.Cha_Con_vSui.2013.NCG_Inner_Fluctuations}\cite{ART.Cha_Con_vSui.2020.NCG_Second_Quantisation}\linebreak\cite{BK.vSui.2015.NCG_AF_Particle_Physics}\cite{BK.Var.2006.NCG_Elements_Short}. This yields our ansatz to study noncommutative gauge theories based on a proposed internalised spectral action \cite{ART.Cha_Con.1996.NCG_Spectral_Action_I}\cite{ART.Cha_Con.1997.NCG_Spectral_Action_II}\cite{ART.Cha_Con_Mar.2007.NCG_Standard_Model_Recovered}\cite{BK.vSui.2015.NCG_AF_Particle_Physics}\cite{BK.Var.2006.NCG_Elements_Short}.\par
However, we defer a detailed discussion to future work as it requires generalisation to dynamic transport distances of states on continuous fields of AF-$C^{*}$-algebras. We still view quantum optimal transport as the pointwise case of a general parametrised one since this strongly motivates non-spatiality. First quantisation considers commutative spectral triples, i.e.~first quantisation of compact spin manifolds \cite{ART.Con.1996.NCG_Reconstruction}. We show quantum optimal transport is transversal to spatial optimal transport in this case. Second quantisation rectifies this by quantising all spatial coordinates. We apply a characterisation in \cite{ART.Cha_Con_vSui.2020.NCG_Second_Quantisation} to obtain sufficient conditions s.t.~the quantum gradients used are infinitesimal evolution of observables at thermal equilibrium determined by KMS-states \cite{BK.Bra.1987.OpAlg_Quantum_StM_II}. Each assumes fixed gauge field \cite{ART.Cha_Con.1996.NCG_Spectral_Action_I}\cite{BK.vSui.2015.NCG_AF_Particle_Physics}\cite{BK.Var.2006.NCG_Elements_Short}. Varying von Neumann entropy \cite{BK.Ohy_Pet.1993.Rel_Ent} of such KMS-states w.r.t.~the canonical trace yields description of the spectral action on gauge fields \cite{ART.Cha_Con.1996.NCG_Spectral_Action_I}\cite{ART.Cha_Con.1997.NCG_Spectral_Action_II}\cite{ART.Cha_Con_Mar.2007.NCG_Standard_Model_Recovered} in terms of quantum statistical mechanics \cite{BK.Bra.1987.OpAlg_Quantum_StM_I}\cite{BK.Bra.1987.OpAlg_Quantum_StM_II} using quantum relative entropy \cite{ART.Cha_Con_vSui.2020.NCG_Second_Quantisation}. Upon passing to second quantisation, we introduce gauge fields as spatial coordinates. We consider it a model, and therefore expect several properties of quantum optimal transport: quantum gradients and thus continuity equations do not use spatial coordinates, we have a description of quantum Laplacians in terms of quantum statistical mechanics, and non-ergodic noncommutative heat semigroups are the rule. We avoid spatial interpretations of the classical case \cite{ART.Dol_Naz_Sav.2009.Generalised_OT}\cite{ART.Lot_Vil.2009.Classical_OT_Ricci_Bounds}, e.g.~as mass transport \cite{BK.Amb_Gig_Sav.2008.Classical_OT_GradFlow}\cite{BK.Vil.2009.OT}, but do require an alternative one for quantum optimal transport.\par


\newpage


The coarse graining process provides such an alternative as it lets us view quantum optimal transport as transport of, suitably general, quantum information. We transport scaling limits of uniformly conditioned spin states encoding sequences of qubits. We avoid spatial interpretations because spin states have physical realisation \cite{ART.Bur_Lad_Nic.2023.QC_Spin_Overview}\cite{BK.Nie_Chu.2000.Quantum_Computation_Information}\cite{ART.DiVi_Loss.1998.Quantum_Information_Physical} s.t.~manipulation of encoded qubits does not consider spatial coordinates. We thereby have non-spatiality, as well as an immediate link to quantum statistical mechanics since information is physical \cite{BK.Cam_Def.2019.Quantum_StM_Information}\cite{ART.DiVi_Loss.1998.Quantum_Information_Physical}\cite{ART.Lan.1961.Information_Physical_I}\cite{ART.Lan.1961.Information_Physical_II}. This link ought to be noticeable if the given quantum system provides physical realisation of a quantum computer \cite{ART.Ash_Geo_Nor.2014.Quantum_Simulation}\cite{BK.Nie_Chu.2000.Quantum_Computation_Information}.\par
Non-ergodicity, defined as complex kernel dimension larger than one for quantum Laplacians, restricts information-bearing degrees of freedom. Since energy functionals are $\Gamma$-limits w.r.t.~the coarse graining process, the latter reduces the AF-$C^{*}$-setting to the finite-dimensional one s.t.~ergodicity is recovered up to a controlled remainder by reducing to accessibility components in the finite-dimensional setting. There may exist\linebreak uncountable infinitely many since sets of states at finite distance have identical fixed parts under noncommutative heat semigroups of quantum Laplacians. Assuming spectral gaps of quantum Laplacians and fixed parts, we use such fixed parts to classify accessibility components of square integrable normal states. Altogether, we study a non-spatial transport of quantum information with restricted information-bearing degrees of freedom. In Chapter \ref{CH.L2W}, we moreover obtain a description of quantum Laplacians in terms of both quantum statistical mechanics and quantum information theory.

\medskip

\noindent\textbf{Structure.} In Section \ref{SEC.QOT_DT}, we discuss quantum optimal transport distances given our noncommutative differential structures. We provide fundamental example classes. In Section \ref{SEC.QOT_AC}, we review support projections of normal states, discuss our use of quantum Fokker-Planck equations, and subsequently study noncommutative heat semigroups of quantum Laplacians. Finally, we classify accessibility components of square integrable normal states. In Section \ref{SEC.QOT_QIT}, we explain the coarse graining process and use it to view quantum optimal transport as transport of quantum information.


\section{Description using dynamic transport distances}\label{SEC.QOT_DT}

Quantum optimal transport requires two notions. First, admissible paths determined by continuity equations. Secondly, energy functionals given by integrating quasi-entropies evaluated on admissible paths. Minimising square roots of energy functionals over all admissible paths for fixed marginals defines quantum optimal transport distances. We show existence of minimising geodesics. Energy functionals are $\Gamma$-limits if restricted to sets of admissible paths with identical interval and marginals, and therefore w.r.t.~the coarse graining process. We formalise the latter as existence of sufficient minimising geodesics approximated in finite dimensions.

\medskip

\noindent\textbf{Structure.} In Subsection \ref{SSEC.QOT_DT_CEN}, we use quasi-entropies to define energy functionals on admissible paths determined by continuity equations. In Subsection \ref{SSEC.QOT_DT_MG}, we discuss quantum optimal transport distances, minimising geodesics and their approximation in finite dimensions. In Subsection \ref{SSEC.QOT_DT_BSP}, we provide all fundamental example classes. An application is given by first and second quantisation of spectral triples.


\subsection{Energy functionals on admissible paths}\label{SSEC.QOT_DT_CEN}

Quantum gradients define continuity equations for states on tracial AF-$C^{*}$-algebras. Note each contains the codomain of the given quantum gradient. Continuity equations define sets of admissible paths. We formulate the latter using Banach dual spaces of Bochner $L^{2}$-spaces. Quasi-entropies define energy functionals by integrating their own evaluation on admissible paths. Altogether, we obtain energy functionals on admissible paths of states on tracial AF-$C^{*}$-algebras.\par
We use compression of quantum gradients and therefore continuity equations to show energy functionals are $\Gamma$-limits. Compressing to induced AF-$C^{*}$-bimodules yields energy functionals on admissible paths of states on generating $C^{*}$-subalgebras. We must initially extend inclusion and restriction maps for Banach dual spaces of tracial AF-$C^{*}$-algebras as per Definition \ref{DFN.AF_Cstar_Trace_Dualisation} to sets of admissible paths. We then compress as above by restricting to induced AF-$C^{*}$-bimodules. Taking limits recovers the initial set of admissible paths. Using the latter, Theorem \ref{THM.Energy_Functional_Representation} shows energy functionals are $\Gamma$-limits if restricted to sets of admissible paths with identical interval and marginals. We thereby extend finite-dimensional approximation of quantum gradients to energy functionals. Standard reference for Bochner $L^{2}$-spaces and their Banach dual spaces is \cite{BK.Ion_Ion.1969.Vector_Valued_Lp_Duals}. Standard reference for $\Gamma$-convergence of functionals is \cite{BK.DalMas.1993.Gamma_Convergence}.


\subsubsection*{Banach dual spaces of Bochner $\mathbf{L}^{2}$-spaces}

Bochner $L^{2}$-spaces have locally convex topological vector spaces as codomains of integration and are not reflexive in general \cite{BK.Ion_Ion.1969.Vector_Valued_Lp_Duals}. We rectify this by considering $w^{*}$-topologies.\par
Let $V$ be a separable Banach space.

\begin{ntn}
Let $I\subset\mathbb{R}$ denote a closed interval. We commonly use $I=[a,b]\subset\mathbb{R}$.
\end{ntn}

We equip all closed intervals $I\subset\mathbb{R}$ with the Lebesgue measure. Radon measures are strictly localisable \cite{BK.Pap.2002.Measures}. Theorem IV.5 in \cite{BK.Ion_Ion.1969.Vector_Valued_Lp_Duals} therefore shows results in \cite{BK.Ion_Ion.1969.Vector_Valued_Lp_Duals} used here apply. A map $h:I\longrightarrow V$ is Bochner measurable if and only if the map $t\mapsto\mu\lc{}h(t)\rc$ is measurable for all $\mu\in V^{*}$. A map $g:I\longrightarrow V^{*}$ is $w^{*}$-measurable if and only if the map $t\mapsto g(t)(v)$ is measurable for all $v\in V$. Separability implies equivalence.

\begin{dfn}\label{DFN.Bochner_L2}
Let $I\subset\mathbb{R}$ be a closed interval.

\begin{itemize}
\item[1)] Set $L^{2}\lc{}I,V\rc{}:=\lset{}h:I\longrightarrow V\ \vset\ \textrm{Bochner measurable},\ \|h\|_{V}^{2}\in L^{1}(I)\rset$. We call $L^{2}\lc{}I,V\rc$ the Bochner $L^{2}$-space of functions from $I$ to $V$. For all $h\in L^{2}\lc{}I,V\rc$, set

\begin{align}\label{EQ.DFN.Bochner_L2_1}
\|h\|_{2}:=\int_{I}\|h(t)\|_{V}^{2}dt.
\end{align}

\begin{reapply}
\end{reapply}

\item[2)] Set $L^{2}\lc{}I,V^{*}\rc_{\w}:=\lset{}g:I\longrightarrow V^{*}\ \vset\ w^{*}\textrm{-measurable},\ \|g\|_{V^{*}}^{2}\in L^{1}(I)\rset$. We call $L^{2}\lc{}I,V^{*}\rc_{\w}$ the $L^{2}$-space of $w^{*}$-functions from $I$ to $V^{*}$. For all $g\in L^{2}\lc{}I,V^{*}\rc_{\w}$, set

\begin{align}\label{EQ.DFN.Bochner_L2_2}
\|g\|_{2}:=\int_{I}\|g(t)\|_{V^{*}}^{2}dt.
\end{align}

\begin{reapply}
\end{reapply}

\end{itemize}
\end{dfn}

\begin{prp}\label{PRP.Bochner_L2}
For all closed intervals $I\subset\mathbb{R}$, we have

\begin{itemize}
\item[1)] $\lc{}L^{2}\lc{}I,V\rc{},\|.\|_{2}\rc$ and $\lc{}L^{2}\lc{}I,V^{*}\rc_{\w},\|.\|_{2}\rc$ are Banach spaces,

\item[2)] $L^{2}\lc{}I,V\rc^{*}=L^{2}\lc{}I,V^{*}\rc_{\w}$.
\end{itemize} 
\end{prp}
\begin{proof}
Let $I\subset\mathbb{R}$ be a closed interval. We use notation in \cite{BK.Ion_Ion.1969.Vector_Valued_Lp_Duals}. Note $L^{2}\lc{}I,V\rc{}=L_{V'}^{2}$ and $L^{2}\lc{}I,V^{*}\rc_{\w}=L_{V'}^{2}[V]$ are Banach spaces. Get $1)$. For all $F\in L^{2}\lc{}I,V\rc^{*}$, Theorem VII.9 in \cite{BK.Ion_Ion.1969.Vector_Valued_Lp_Duals} and its immediate corollary show there exists unique $g_{F}\in L^{2}\lc{}I,V^{*}\rc_{\w}$ s.t.~

\begin{align}\label{EQ.PRP.Bochner_L2_1}
F(h)=\int_{I}g_{F}(t)(h(t))dt
\end{align}

\noindent for all $h\in L^{2}\lc{}I,V\rc$. Equation \ref{EQ.DFN.Bochner_L2_1} and Equation \ref{EQ.DFN.Bochner_L2_2} further imply

\begin{align}\label{EQ.PRP.Bochner_L2_2}
\dblv{}F\dblv_{L^{2}\lc{}I,V\rc^{*}}=\sup_{\substack{h\in L^{2}(I,V),\\ \|h\|_{2}\leq 1}}\hspace{0.025cm} \bbbabsv{1}{\int_{I}g_{F}(t)(h(t))dt}=\int_{I}\dblv{}g_{F}(t)\dblv_{V^{*}}^{2}dt=\dblv{}g_{F}\dblv_{2}
\end{align}

\noindent in each case. Therefore, $L^{2}\lc{}I,V\rc^{*}=L^{2}\lc{}I,V^{*}\rc_{\w}$. Get $2)$.
\end{proof}

\begin{rem}\label{REM.Weak_Metrisability}
Let $V$ be a separable Banach space and $K\subset V^{*}$ norm bounded. Given $\{v_{n}\}_{n\in\mathbb{N}}\subset V\setminus\lset{}0\rset$ with $\|.\|_{V}$-dense linear span, set $d\lc\rho,\rho'\rc{}:=\sum_{n=1}^{\infty}2^{-n}\| v_{n}\|_{V}^{-1}\absv{1}{\rho\lc{}v_{n}\rc{}-\rho'\lc{}v_{n}\rc{}}$ for all $\rho,\rho'\in K$. This defines a distance metricising the $w^{*}$-topology on $K$.
\end{rem}


\subsubsection*{Admissible paths determined by continuity equations}

Definition \ref{DFN.Continuity_Equation}, in particular Equation \ref{EQ.DFN.Continuity_Equation_1}, gives continuity equations. Definition \ref{DFN.Admissible_Paths} gives admissible paths determined by continuity equations. Admissible paths lie in state spaces of tracial AF-$C^{*}$-algebras as per Definition \ref{DFN.Cstar_Trace_Abstract_State_Space}. Proposition \ref{PRP.Continuity_Equation_Mass_Preservation} shows norm-preservation.\par
Let $(\phi,\bpsi,\gamma,\nabla)$ be noncommutative differential structure for tracial AF-$C^{*}$-algebras $(A,\tau)$ and $(B,\omega)$ in $\lc{}f,\theta\rc$-setting.

\begin{dfn}\label{DFN.Continuity_Equation}
Let $I\subset\mathbb{R}$ be a closed interval.

\begin{itemize}
\item[1)] We say that $\mu:I\longrightarrow A_{+}^{*}$ is weakly absolutely continuous if $t\mapsto\mu(t)(x)$ is absolutely continuous for all $x\in A_{0}$.

\item[2)] Let $\mu:I\longrightarrow A_{+}^{*}$ be weakly absolutely continuous and $w\in L^{2}\lc{}I,B^{*}\rc_{\w}$. The pair $(\mu,w)$ satisfies the continuity equation for $\nabla$ on $I$ if

\begin{align}\label{EQ.DFN.Continuity_Equation_1}
\frac{d}{dt}\mu(t)(x)=w(t)\lc\nabla x\rc{}
\end{align}

\begin{reapply}
\end{reapply}

\noindent for all $x\in A_{0}$ and a.e.~$t\in I$.
\end{itemize}
\end{dfn}


\pagebreak


\begin{prp}\label{PRP.Continuity_Equation_Mass_Preservation}
Let $\mu:I\longrightarrow A_{+}^{*}$ be weakly absolutely continuous and $w\in L^{2}\lc{}I,B^{*}\rc_{\w}$. If $(\mu,w)$ satisfies the continuity equation for $\nabla$ on $I$, then

\begin{align}\label{EQ.PRP.Continuity_Equation_Mass_Preservation_1}
\big\|\mu(t)\big\|_{A^{*}}=\big\|\mu(s)\big\|_{A^{*}}
\end{align}

\noindent for all $t,s\in I$.
\end{prp}
\begin{proof}
For all $j\in\mathbb{N}$, note $\nabla 1_{A_{j}}=0$ the Leibniz rule. Thus $\frac{d}{dr}\mu(r)(1_{A_{j}})=0$ for a.e.~$r\in I$ for all $j\in\mathbb{N}$, hence 

\begin{align}\label{EQ.PRP.Continuity_Equation_Mass_Preservation_2}
\mu(t)(1_{A_{j}})=\mu(s)(1_{A_{j}})  
\end{align}

\noindent for all $t,s\in I$ and $j\in\mathbb{N}$. Set $\mu_{j}(t):=\mu(t)_{j}=\mu(t)\vert_{A_{j}}$ in each case. Positivity ensures

\begin{align}\label{EQ.PRP.Continuity_Equation_Mass_Preservation_3}
\big\|\mu_{j}(t)\big\|_{A^{*}}=\mu_{j}(t)(1_{A_{j}})=\mu(t)(1_{A_{j}})
\end{align}

\noindent in each case. Using $1.1)$ in Proposition \ref{PRP.AF_Cstar_Trace_Dualisation_II}, we see Equation \ref{EQ.PRP.Continuity_Equation_Mass_Preservation_2} and Equation \ref{EQ.PRP.Continuity_Equation_Mass_Preservation_3} imply Equation \ref{EQ.PRP.Continuity_Equation_Mass_Preservation_1} at once.
\end{proof}

\begin{dfn}\label{DFN.Admissible_Paths}
Let $\overline{\SII(A)}$ denote the $w^{*}$-closure of $\SII(A)\subset A_{+}^{*}$.

\begin{itemize}
\item[1)]  Let $I=[a,b]\subset\mathbb{R}$. Set

\begin{itemize}
\item[1.1)] $\AC(I,\overline{\SII(A)}):=\big\{\hspace{0.025cm} \mu:I\longrightarrow\overline{\SII(A)}\ \vset\ \mu\ \textrm{is weakly absolutely continuous}\hspace{0.025cm} \big\}$, \phantom{\big)}

\item[1.2)] $\AC\lc{}I,\SII(A)\rc{}:=\big\{\hspace{0.025cm} \mu\in\AC(I,\overline{\SII(A)})\ \vset\ \im\mu\subset\SII(A)\hspace{0.025cm} \big\}$. \phantom{\big)}
\end{itemize}

\begin{reapply}
\end{reapply}

\item[2)] We say that $(\mu,w)\in\AC\lc{}[a,b],\SII(A)\rc\times L^{2}([a,b],B^{*})_{\w}$ is an admissible path if $(\mu,w)$ satisfies the continuity equation for $\nabla$ on $[a,b]$. We further call $\mu\lc{}a\rc{},\mu\lc{}b\rc\in\SII(A)$ the marginals of $(\mu,w)$, resp.~$\mu$.

\item[3)] For all $\mu^{0},\mu^{1}\in\SII(A)$, let $\Admab\lc\mu^{0},\mu^{1}\rc$ be the set of all admissible paths defined on $[a,b]\subset\mathbb{R}$ with marginals $\mu^{0}$ and $\mu^{1}$. Set

\begin{itemize}
\item[3.1)] $\Adm\lc\mu^{0},\mu^{1}\rc{}:=\bigcup_{[a,b]\subset\mathbb{R}}\Admab\lc\mu^{0},\mu^{1}\rc$ for all $\mu^{0},\mu^{1}\in\SII(A)$, \phantom{\bigg)}

\item[3.2)] $\Admab:=\bigcup_{\mu^{0},\mu^{1}\in\SII(A)}\Admab\lc\mu^{0},\mu^{1}\rc$ for all $[a,b]\subset\mathbb{R}$, \phantom{\bigg)}

\item[3.3)] $\Adm:=\bigcup_{[a,b]\subset\mathbb{R}}\bigcup_{\mu^{0},\mu^{1}\in\SII(A)}\Admab\lc\mu^{0},\mu^{1}\rc$. \phantom{\bigg)}
\end{itemize}

\begin{reapply}
\end{reapply}

\end{itemize}
\end{dfn}

\begin{ntn}\label{NTN.Admissible_Paths}
For all $j\in\mathbb{N}$, we use $\Adm_{j}$ when denoting sets of admissible paths in Definition \ref{DFN.Admissible_Paths} for the induced noncommutative differential structure $\lc\phi_{j},\bpsi_{j},\gamma_{j},\nabla_{\hspace{-0.055cm} j}\rc$.
\end{ntn}

\begin{rem}\label{REM.Admissible_Paths}
Let $I\subset\mathbb{R}$ be a closed interval and $j\in\mathbb{N}$. We have $A_{j}\cong A_{j}^{*}$ and $B_{j}\cong B_{j}^{*}$ via musical isomorphisms and therefore

\begin{align}\label{EQ.REM.Admissible_Paths_1}
L^{2}(I,A_{j}^{*})_{\w}\cong L^{2}(I,A_{j}),\ L^{2}(I,B_{j}^{*})_{\w}\cong L^{2}(I,B_{j}).
\end{align}

\noindent Each Bochner $L^{2}$-space in Equation \ref{EQ.REM.Admissible_Paths_1} is norm equivalent to the respective Hilbert space of square integrable functions. Up to musical isomorphisms applied to codomains of integration, each $L^{2}$-space of $w^{*}$-functions in Equation \ref{EQ.REM.Admissible_Paths_1} is therefore likewise norm equivalent to such a Hilbert space of square integrable functions.
\end{rem}

Definition \ref{DFN.Admissible_Paths_Convergence} gives the canonical topology on sets of admissible paths alongside a related notion of convergence for the latter. Proposition \ref{PRP.Admissible_Paths_Convergence} collects properties. Let $[a,b]\subset\mathbb{R}$. Since $\AC([a,b],\overline{\SII(A)})\subset L^{2}([a,b],A^{*})_{\w}$ up to null sets, we obtain the canonical inclusion 

\begin{align}\label{EQ.SSEC.QOT_DT_CEN_Fluff}
\Admab\subset L^{2}([a,b],A^{*})_{\w}\times L^{2}([a,b],B^{*})_{\w}.
\end{align}

\noindent The relative topology on $\Admab$ w.r.t.~the $w^{*}$-topology on $L^{2}([a,b],A^{*})_{\w}\times L^{2}([a,b],B^{*})_{\w}$ given by the above canonical inclusion is called the relative $w^{*}$-topology.\par
We define a second topology on $\Admab$ by equipping

\begin{align}\label{EQ.SSEC.QOT_DT_CEN_1}
\overline{\SII(A)}^{[a,b]}:=\prod_{t\in [a,b]}\overline{\SII(A)}   
\end{align}

\noindent with the product topology given by the $w^{*}$-topology on $\overline{\SII(A)}$. Pointwise convergence in $w^{*}$-topology is convergence in the product topology. We further consider $w^{*}$-topology on $L^{2}\lc{}[a,b],B\rc^{*}=L^{2}([a,b],B^{*})_{\w}$ as per $2)$ in Proposition \ref{PRP.Bochner_L2}. The relative topology given by the canonical inclusion

\begin{align}\label{EQ.SSEC.QOT_DT_CEN_2}
\Admab\subset\overline{\SII(A)}^{[a,b]}\times L^{2}([a,b],B^{*})_{\w}
\end{align}

\noindent is called the canonical topology on $\Admab$.

\begin{dfn}\label{DFN.Admissible_Paths_Convergence}
For all $[a,b]\subset\mathbb{R}$, the relative topology as per Equation \ref{EQ.SSEC.QOT_DT_CEN_2} is called the canonical topology on $\Admab$. We say that $(\mu^{n},w^{n})_{n\in\mathbb{N}}\subset\Admab$ converges to $(\mu,w)\in\Admab$ if

\begin{itemize}
\item[1.1)] $\mu(t)=w^{*}$-$\lim_{n\in\mathbb{N}}\mu^{n}(t)$ in $\SII(A)$ for all $t\in [a,b]$,

\item[1.2)] $w=w^{*}$-$\lim_{n\in\mathbb{N}}w^{n}$ in $L^{2}([a,b],B^{*})_{\w}$.
\end{itemize}

\begin{reapply}
\end{reapply}

\noindent We further write $(\mu,w)=\lim_{n\in\mathbb{N}}\hspace{0.025cm} (\mu^{n},w^{n})$ in $\Admab$.
\end{dfn}

\begin{prp}\label{PRP.Admissible_Paths_Convergence}
Let $(\mu^{n},w^{n})_{n\in\mathbb{N}}\subset\Admab$.

\begin{itemize}
\item[1)] Let $(\mu,w)\in\AC([a,b],\overline{\SII(A)})\times L^{2}([a,b],B^{*})_{\w}$ s.t.~

\begin{itemize}
\item[1.1)] $\mu(t)=w^{*}$-$\lim_{n\in\mathbb{N}}\mu^{n}(t)$ for all $t\in [a,b]$,

\item[1.2)] $w=w^{*}$-$\lim_{n\in\mathbb{N}}w^{n}$.
\end{itemize}

\begin{reapply}
\end{reapply}

\noindent If there exists $t_{0}\in [a,b]$ s.t.~$\|\mu(t_{0})\|_{A^{*}}=1$, then $(\mu,w)\in\Admab$.

\item[2)] If $(\mu,w)=\lim_{n\in\mathbb{N}}\hspace{0.025cm} (\mu^{n},w^{n})$ in $\Admab$, then $(\mu,w)=w^{*}$-$\lim_{n\in\mathbb{N}}\hspace{0.025cm} (\mu^{n},w^{n})$.
\end{itemize}
\end{prp}
\begin{proof}
We show $1)$. Assume its setting. For all $x\in A_{0}$, the map $t\mapsto g(t):=\nabla x$ defined on $[a,b]$ lies in $L^{2}\lc{}[a,b],B\rc$ by locality if we identify as per Remark \ref{REM.Admissible_Paths}. For all $x\in A_{0}$ and $h\in (0,1)$, we apply the continuity equation in order to rewrite the difference quotient

\begin{align}\label{EQ.PRP.Admissible_Paths_Convergence_1}
\frac{1}{h}\lc\mu\lc{}t+h\rc{}(x)-\mu(0)(x)\rc{}=\lim_{n\in\mathbb{N}}\hspace{0.025cm} \frac{1}{h}\int_{0}^{t+h}\frac{d}{ds}w^{n}(s)\lc\nabla x\rc{}ds=\frac{1}{h}\int_{0}^{t+h}w(s)\lc\nabla x\rc{}ds.
\end{align}

\noindent Letting $h\rightarrow 0$ in Equation \ref{EQ.PRP.Admissible_Paths_Convergence_1} shows $(\mu,w)$ satisfies the continuity equation for $\nabla$ on $[a,b]$. Proposition \ref{PRP.Continuity_Equation_Mass_Preservation} shows norm-preservation. We see $1)$ at once. Moreover, standard arguments show $2)$ by dominated convergence.
\end{proof}

Definition \ref{DFN.AF_Cstar_Trace_Dualisation_Paths} extends restriction maps in Definition \ref{DFN.AF_Cstar_Trace_Dualisation} to all paths in Banach dual spaces of tracial AF-$C^{*}$-algebras. Proposition \ref{PRP.AF_Cstar_Trace_Dualisation_Admissible_Paths} further extends inclusion and restriction maps to sets of admissible paths. Restricting paths rescales norm.

\begin{dfn}\label{DFN.AF_Cstar_Trace_Dualisation_Paths}
Let $\mathcal{A}$ be a tracial AF-$C^{*}$-algebra, $I\subset\mathbb{R}$ a closed interval and $j\in\mathbb{N}$.
 
\begin{itemize}
\item[1)] For all $\rho\in\mathcal{A}_{+}^{*}$, set

\begin{align*}
\bar{\rho}_{j}:=
\begin{cases}
\rho(1_{\mathcal{A}_{j}})^{-1}\rho_{j} & \If\ \rho(1_{\mathcal{A}_{j}})\neq 0, \\
0 & \Else.
\end{cases}
\end{align*}

\begin{reapply}
\end{reapply}

\item[2)] Let $\rho:I\longrightarrow\mathcal{A}_{+}^{*}$ be defined for a.e.$~t\in I$. We define $\rho_{j}:I\longrightarrow\mathcal{A}_{j}^{*}$ and $\bar{\rho}_{j}:I\longrightarrow\mathcal{A}_{j}^{*}$ by setting

\begin{align}\label{EQ.DFN.AF_Cstar_Trace_Dualisation_Paths_1}
\rho_{j}(t):=\rho(t)_{j},\ \bar{\rho}_{j}(t):=\overline{\rho(t)}_{j}
\end{align}

\begin{reapply}
\end{reapply}

\noindent for a.e.~$t\in I$.
\end{itemize}
\end{dfn}

\begin{rem}\label{REM.AF_Cstar_Trace_Dualisation_Paths}
For all $\rho\in\mathcal{A}_{+}^{*}$, we have $\dblv{}\rho_{j}\dblv_{A^{*}}=\rho(1_{A_{j}})$ for all $j\in\mathbb{N}$ by positivity. We obtain $\rho(1_{A_{j}})\neq 0$ for a.e.~$j\in\mathbb{N}$ by $1)$ in Proposition \ref{PRP.AF_Cstar_Trace_Dualisation_II}. For all $\mu\in\SII(A)$, we have $\bar{\mu}_{j}\in\mathcal{S}(A_{j})$ if and only if $\mu_{j}\neq 0$. We use this throughout our discussion.
\end{rem}

Following Remark \ref{REM.AF_Cstar_Trace_Dualisation_Admissible_Paths}, we assume strictly positive norm for at least one marginal if we apply restriction maps as per Proposition \ref{PRP.AF_Cstar_Trace_Dualisation_Admissible_Paths}. A more rigorous but cumbersome notation may further include marginals in sets of admissible paths.

\begin{prp}\label{PRP.AF_Cstar_Trace_Dualisation_Admissible_Paths}
For all $[a,b]\subset\mathbb{R}$ and $j\leq k$ in $\mathbb{N}$, we define

\begin{itemize}
\item[1)] the $j$-th inclusion and restriction

\begin{align}\label{EQ.PRP.AF_Cstar_Trace_Dualisation_Admissible_Paths_1}
\incj:\Admabj\longrightarrow\Admab,\ \resj:\Admab\longrightarrow\Admabj
\end{align}

\begin{reapply}
\end{reapply}

\noindent by setting

\begin{align}\label{EQ.PRP.AF_Cstar_Trace_Dualisation_Admissible_Paths_2}
\incj(\mu,w):=(\mu,w),\ \resj(\mu,w):=\lc\bar{\mu}_{j},\mu\lc{}a\rc{}(1_{A_{j}})^{-1}w_{j}\rc{}
\end{align}

\begin{reapply}
\end{reapply}

\noindent for all $(\mu,w)\in\Admabj$, resp.~$(\mu,w)\in\Admab$.

\item[2)] the $kj$-inclusion and $jk$-restriction

\begin{align}\label{EQ.PRP.AF_Cstar_Trace_Dualisation_Admissible_Paths_3}
\inckj:\Admabj\longrightarrow\Admabk,\ \resjk:\Admabk\longrightarrow\Admabj   
\end{align}

\begin{reapply}
\end{reapply}

\noindent by setting

\begin{align}\label{EQ.PRP.AF_Cstar_Trace_Dualisation_Admissible_Paths_4}
\inckj(\mu,w):=(\mu,w),\ \resjk(\mu,w):=\lc\bar{\mu}_{j},\mu\lc{}a\rc{}(1_{A_{j}})^{-1}w_{j}\rc{}
\end{align}

\begin{reapply}
\end{reapply}

\noindent for all $(\mu,w)\in\Admabj$, resp.~$(\mu,w)\in\Admabk$.
\end{itemize}
\end{prp}
\begin{proof}
We show $1)$, i.e.~the case of $k=\infty$. We obtain $2)$ by analogous argument for $k<\infty$. We know $w^{*}$-continuity of inclusion and restriction maps by $1)$ in Proposition \ref{PRP.AF_Cstar_Trace_Dualisation_I}. Upon identifying as per Remark \ref{REM.Admissible_Paths}, $\incj$ maps to $\AC\lc{}I,\SII(A)\rc\times L^{2}\lc{}I,B^{*}\rc_{\w}$ and $\resj$ to $\AC\lc{}I,\SII(A_{j})\rc\times L^{2}\lc{}I,B_{j}\rc$. Using the latter and Proposition \ref{PRP.Wstar_Derivation_QG_I}, we directly verify all claimed continuity equations.
\end{proof}

\begin{rem}\label{REM.AF_Cstar_Trace_Dualisation_Admissible_Paths}
If $(\mu,w)$ satisfies the continuity equation for $\nabla$ on $I$, then $\bar{\mu}(t)_{j}=\bar{\mu}(0)_{j}$ for all $t\in I$ and $j\in\mathbb{N}$ by Proposition \ref{PRP.Continuity_Equation_Mass_Preservation}. Non-trivial restriction requires $\mu(0)_{j}\neq 0$ in each case. If we apply restriction maps as per Proposition \ref{PRP.AF_Cstar_Trace_Dualisation_Admissible_Paths}, then we either assume $\mu(0)_{j}\neq 0$ for all $j\in\mathbb{N}$ as part of a statement itself or we assume it implicitly without loss of generality since Remark \ref{REM.AF_Cstar_Trace_Dualisation_Paths} ensures $\mu(0)_{j}\neq 0$ for a.e.~$j\in\mathbb{N}$.
\end{rem}


\subsubsection*{Energy functionals from quasi-entropies}

Definition \ref{DFN.Energy_Functional} describes energy functionals given by integrating quasi-entropies evaluated on admissible paths. Note Remark \ref{REM.Energy_Functional}. Definition \ref{DFN.Energy_Functional_Extension} gives an a priori different description. They coincide on admissible paths. Proposition \ref{PRP.Energy_Functional_Restriction} extends results in Theorem \ref{THM.QE_AF} concerning inclusion and restriction maps to energy functionals.\par
Let $(\phi,\bpsi,\gamma,\nabla)$ be noncommutative differential structure for tracial AF-$C^{*}$-algebras $(A,\tau)$ and $(B,\omega)$ in $\lc{}f,\theta\rc$-setting.

\begin{dfn}\label{DFN.Energy_Functional} 
We define the energy functional $E^{f,\theta}:\Adm\longrightarrow [0,\infty]$ by setting

\begin{align}\label{EQ.DFN.Energy_Functional_1}
E^{f,\theta}(\mu,w):=\int_{a}^{b}\mathcal{I}^{f,\theta}\lc\mu(t),\mu(t),w(t)\rc{}dt
\end{align}

\noindent for all $[a,b]\subset\mathbb{R}$ and $(\mu,w)\in\Admab$.
\end{dfn}

\begin{ntn}\label{NTN.Energy_Functional} 
For all $j\in\mathbb{N}$, let $E_{j}^{f,\theta}$ denote energy functional in Definition \ref{DFN.Energy_Functional} for the induced noncommutative differential structure $\lc\phi_{j},\bpsi_{j},\gamma_{j},\nabla_{\hspace{-0.055cm} j}\rc$.
\end{ntn}

\begin{rem}\label{REM.Energy_Functional} 
For all $j\in\mathbb{N}$, Equation \ref{EQ.DFN.Energy_Functional_1} is

\begin{align}\label{EQ.REM.Energy_Functional_1}
E_{j}^{f,\theta}(\mu,w)=\int_{a}^{b}\mathcal{I}_{j}^{f,\theta}\lc\mu(t),\mu(t),w(t)\rc{}dt
\end{align}

\noindent for all $[a,b]\subset\mathbb{R}$ and $(\mu,w)\in\Admabj$.
\end{rem}

\begin{prp}\label{PRP.Energy_Functional_Restriction}
Let $[a,b]\subset\mathbb{R}$ and $j\leq k$ in $\mathbb{N}$.

\begin{itemize}
\item[1)] For all $(\mu,w)\in\Admabj\lc\mu^{0},\mu^{1}\rc$, we have

\begin{align}\label{EQ.PRP.Energy_Functional_Restriction_1}
E_{j}^{f,\theta}(\mu,w)=E_{k}^{f,\theta}\lc\inckj(\mu,w)\rc{}=E^{f,\theta}\lc\incj(\mu,w)\rc{}.
\end{align}

\begin{reapply}
\end{reapply}

\item[2)] Assume $\mu^{0}(1_{A_{j}})\neq 0$ in all statements below.

\begin{itemize}\label{EQ.PRP.Energy_Functional_Restriction_2}
\item[2.1)] For all $(\mu,w)\in\Admab\lc\mu^{0},\mu^{1}\rc$, we have

\begin{align}
E_{j}^{f,\theta}\lc\resj(\mu,w)\rc\leq\mu^{0}(1_{A_{j}})^{-1}E^{f,\theta}(\mu,w).
\end{align}

\begin{reapply}
\end{reapply}

\item[2.2)] For all $(\mu,w)\in\Admabk\lc\mu^{0},\mu^{1}\rc$, we have

\begin{align}\label{EQ.PRP.Energy_Functional_Restriction_3}
E_{j}^{f,\theta}\lc\resjk(\mu,w)\rc\leq\mu^{0}(1_{A_{j}})^{-1}E_{k}^{f,\theta}(\mu,w).
\end{align}

\begin{reapply}
\end{reapply}

\end{itemize}

\begin{reapply}
\end{reapply}

\end{itemize}
\end{prp}
\begin{proof}
Equation \ref{EQ.REM.Energy_Functional_1} shows $1)$ by $2)$ in Theorem \ref{THM.QE_AF}. We show $2)$. Assume its setting. For all $\mu,\eta\in A_{+}^{*}$, $w\in B^{*}$ and $\lambda\geq 0$, get $\mathcal{I}^{f,\theta}\lc\lambda\mu,\lambda\eta,\lambda w\rc{}=\lambda\mathcal{I}^{f,\theta}(\mu,\eta,w)$ by construction of quasi-entropies. For all $t\in [0,1]$, Equation \ref{EQ.PRP.Continuity_Equation_Mass_Preservation_2} shows $\mu(t)(1_{A_{j}})=\mu^{0}(1_{A_{j}})\neq 0$. Using the latter, we obtain $2)$ by $3)$ in Theorem \ref{THM.QE_AF}.
\end{proof}

Proposition \ref{PRP.Energy_Functional_Reparametrisation} gives a change of variables formula for energy functionals. For this, Remark \ref{REM.Change_Of_Variables} states a general one for reparametrisations of measurable functions \lc{}cf.~Corollary 6 to Theorem 3 in \cite{ART.Ser_Var.1969.General_Chain_Rule}\rc{}. We commonly use affine transformations as per Remark \ref{REM.Energy_Functional_Reparametrisation}. Proposition \ref{PRP.Energy_Functional_Lipschitz_Sq} extends $5)$ in Theorem \ref{THM.QE_AF} to energy functionals and derives Lipschitz continuity. 

\begin{rem}\label{REM.Change_Of_Variables}
Let $g:[a,b]\longrightarrow\mathbb{R}$ be Lebesgue integrable. If $\varphi:[c,d]\longrightarrow [a,b]$ is monotone and absolutely continuous, then $\dot{\varphi}\cdot \lc{}g\circ\varphi\rc$ is Lebesgue integrable and we have 

\begin{align}\label{EQ.REM.Change_Of_Variables_1}
\int_{\varphi\lc{}c\rc{}}^{\varphi\lc{}d\rc{}}g(t)dt=\int_{a}^{b}\dot{\varphi}(t)g\lc\varphi(t)\rc{}dt.
\end{align}
\end{rem}

\begin{prp}\label{PRP.Energy_Functional_Reparametrisation}
Let $\varphi:[c,d]\longrightarrow [a,b]$ be monotone and absolutely continuous with $\dot{\varphi}(t)\neq 0$ for a.e.~$t\in [c,d]$. If $(\mu,w)\in\Admab\lc\mu^{0},\mu^{1}\rc$, then $\lc\mu\circ\varphi,\dot{\varphi}\cdot \lc{}w\circ\varphi\rc\rc\in\Admcd\lc\mu^{0},\mu^{1}\rc$ and we have

\begin{align}\label{EQ.PRP.Energy_Functional_Reparametrisation_1}
E^{f,\theta}(\mu,w)=\int_{c}^{d}\dot{\varphi}(t)^{-1}\mathcal{I}^{f,\theta}\lc\mu\lc\varphi(t)\rc{},\mu\lc\varphi(t)\rc{},\dot{\varphi}(t)w\lc\varphi(t)\rc\rc{}dt.
\end{align}
\end{prp}
\begin{proof}
Since $\varphi$ is monotone and $t\mapsto\mu(t)(x)$ is absolutely continuous for all $x\in A_{0}$, the chain rule holds for $\mu\circ\varphi$ upon evaluation by Theorem 2 and Corollary 4 in \cite{ART.Ser_Var.1969.General_Chain_Rule}. Thus $\lc\mu\circ\varphi,\dot{\varphi}\cdot \lc{}w\circ\varphi\rc\rc$ satisfies the continuity equation for $\nabla$ on $[c,d]$. All remaining properties of admissible paths are inherited. For all $\mu,\eta\in A_{+}^{*}$, $w\in B^{*}$ and $\lambda\geq 0$, get $\mathcal{I}^{f,\theta}\lc\mu,\eta,\lambda w\rc{}=\lambda^{2}\mathcal{I}^{f,\theta}(\mu,\eta,w)$ by construction of quasi-entropies. Using the latter, Equation \ref{EQ.REM.Change_Of_Variables_1} shows Equation \ref{EQ.PRP.Energy_Functional_Reparametrisation_1} immediately.
\end{proof}

\begin{rem}\label{REM.Energy_Functional_Reparametrisation}
Let $[a,b],[c,d]\subset\mathbb{R}$ s.t.~$a\neq b,c\neq d$. We define monotone and absolutely continuous homeomorphism $\varphi:[c,d]\longrightarrow [a,b]$ by setting 

\begin{align}\label{EQ.REM.Energy_Functional_Reparametrisation_1}
\varphi(t):=\frac{b-a}{d-c}\lc{}t-c\rc{}+a    
\end{align}

\noindent for all $t\in [c,d]$. Using Proposition \ref{PRP.Energy_Functional_Reparametrisation}, Equation \ref{EQ.PRP.Energy_Functional_Reparametrisation_1} shows

\begin{align}\label{EQ.REM.Energy_Functional_Reparametrisation_2}
E^{f,\theta}(\mu,w)=\frac{d-c}{b-a}E^{f,\theta}\bigg(\mu\circ\varphi,\frac{b-a}{d-c}\lc{}w\circ\varphi\rc\bigg).
\end{align}
\end{rem}


\pagebreak


\begin{prp}\label{PRP.Energy_Functional_Lipschitz_Sq}
Let $[a,b]\subset\mathbb{R}$.

\begin{itemize}
\item[1)] For all $(\mu,w)\in\Admab$, we have

\begin{align}\label{EQ.PRP.Energy_Functional_Lipschitz_Sq_1}
\|w\|_{L^{2}([a,b],B^{*})_{\w}}^{2}\leq E^{f,\theta}(\mu,w)\cdot 2^{-\theta}\lc\|\phi\|_{1}^{\theta}+\|\bpsi\|_{1}^{\theta}\rc\cdot \|\omega\|^{1-\theta}.
\end{align}

\begin{reapply}
\end{reapply}

\item[2)] For all $(\mu,w)\in\Admab$, $x\in A_{0}$ and $t,s\in [a,b]$, we have

\begin{align}\label{EQ.PRP.Energy_Functional_Lipschitz_Sq_2}
\babsv{1}{\lc\mu(t)-\mu(s)\rc{}(x)}^{2}\leq \absv{1}{t-s}\cdot E^{f,\theta}(\mu,w)\cdot 2^{-\theta}\lc\|\phi\|_{1}^{\theta}+\|\bpsi\|_{1}^{\theta}\rc\cdot \|\omega\|^{1-\theta}\cdot \|\nabla x\|_{B}^{2}.
\end{align}

\begin{reapply}
\end{reapply}

\end{itemize}
\end{prp}
\begin{proof}
Note Equation \ref{EQ.DFN.Bochner_L2_2} ensures $1)$ follows by $5)$ in Theorem \ref{THM.QE_AF}. We show $2)$. For all $(\mu,w)\in\Admab$, $x\in A_{0}$ and $t,s\in [a,b]$, we use the continuity equation and apply H\"older in order to estimate

\begin{align*}
\babsv{1}{\lc\mu(t)-\mu(s)\rc{}(x)} & = \bbbabsv{1}{\int_{s}^{t}\frac{d}{dr}\mu(r)(x)dr} \phantom{\bigg)} \\
& \leq \bbbabsv{1}{\int_{s}^{t}\|w(r)\|_{B^{*}}\|\nabla x\|_{B}dr} \phantom{\bigg)} \\
& \leq \sqrt{\absv{1}{t-s}}\cdot \|w\|_{L^{2}([a,b],B^{*})_{\w}}\cdot \|\nabla x\|_{B}. \phantom{\bigg)}
\end{align*}

\noindent We obtain $2)$ by applying Equation \ref{EQ.PRP.Energy_Functional_Lipschitz_Sq_1} to the above calculation.
\end{proof}

Definition \ref{DFN.Energy_Functional_Extension} gives an a priori different, as well as more general, description of energy functionals than Definition \ref{DFN.Energy_Functional} for a larger domain. Lemma \ref{LEM.Energy_Functional_I} shows both descriptions coincide on admissible paths. Moreover, extensions of energy functionals as per Definition \ref{DFN.Energy_Functional_Extension} are l.s.c~in $w^{*}$-topology. Lemma \ref{LEM.Energy_Functional_II} leverages the latter in order to show l.s.c.~of energy functionals w.r.t.~convergence in canonical topology, and further ensures the direct method in the calculus of variations \cite{BK.DalMas.1993.Gamma_Convergence}\cite{BK.Eva.2010.Partial_Differential_Equations} applies.

\begin{dfn}\label{DFN.Energy_Functional_Extension}
We define $\mathbf{E}^{f,\theta}:\bigcup_{[a,b]\subset\mathbb{R}}L^{2}([a,b],A^{*})_{\w}\times L^{2}([a,b],B^{*})_{\w}\longrightarrow [0,\infty]$ by setting

\begin{align}\label{EQ.DFN.Energy_Functional_Extension_1}
\mathbf{E}^{f,\theta}(\mu,w):=\sup_{j\in\mathbb{N}}\hspace{0.025cm} \int_{a}^{b}\mathcal{I}_{j}^{f,\theta}\lc\mu_{j}(t),\mu_{j}(t),w_{j}(t)\rc{}dt
\end{align}

\noindent for all $[a,b]\subset\mathbb{R}$ and $(\mu,w)\in L^{2}([a,b],A^{*})_{\w}\times L^{2}([a,b],B^{*})_{\w}$.
\end{dfn}

For all $[a,b]\subset\mathbb{R}$, the inclusion in Equation \ref{EQ.SSEC.QOT_DT_CEN_2} extends to 

\begin{align}\label{EQ.SSEC.QOT_DT_CEN_3}
\Admab\subset L^{2}([a,b],A^{*})_{\w}\times L^{2}([a,b],B^{*})_{\w}.  
\end{align}

\noindent Thus Equation \ref{EQ.SSEC.QOT_DT_CEN_3} shows $\Adm\subset\bigcup_{[a,b]\subset\mathbb{R}}L^{2}([a,b],A^{*})_{\w}\times L^{2}([a,b],B^{*})_{\w}$, hence we have functional $\mathbf{E}^{f,\theta}:\Adm\longrightarrow [0,\infty]$ by restricting to $\Adm$.

\begin{lem}\label{LEM.Energy_Functional_I}
For all $[a,b]\subset\mathbb{R}$, $\restr{0.925}{\mathbf{E}^{f,\theta}}{\Admab}$ is l.s.c.~in $w^{*}$-topology. We further have 

\begin{align}\label{EQ.LEM.Energy_Functional_I_1}
E^{f,\theta}=\restr{0.925}{\mathbf{E}^{f,\theta}}{\Adm}.
\end{align}
\end{lem}
\begin{proof}
Let $[a,b]\subset\mathbb{R}$. We show $\restr{0.925}{\mathbf{E}^{f,\theta}}{\Admab}$ is l.s.c.~in $w^{*}$-topology. For all $j\in\mathbb{N}$, we show

\begin{align}\label{EQ.LEM.Energy_Functional_I_2}
(\mu,w)\mapsto\int_{a}^{b}\mathcal{I}_{j}^{f,\theta}\lc\mu_{j}(t),\mu_{j}(t),w_{j}(t)\rc{}dt
\end{align}

\noindent is l.s.c.~in $w^{*}$-topology. Compactness shows pointwise restriction yields $w^{*}$-continuous map from $L^{2}([a,b],A^{*})\times L^{2}([a,b],B^{*})$ to $L^{2}\lc{}[a,b],A_{j}\rc\times L^{2}\lc{}[a,b],B_{j}\rc$. We reduce l.s.c.~in $w^{*}$-topology to the finite-dimensional setting. Assume $A$ and $B$ are finite-dimensional. Note $\mathcal{I}^{f,\theta}$ is jointly convex and l.s.c.~in $w^{*}$-topology by $1)$ in Theorem \ref{THM.QE_AF}. Further note joint convexity implies $\mathbf{E}^{f,\theta}$ is jointly convex. Following Remark \ref{REM.Admissible_Paths}, it suffices to show sequential l.s.c.~in norm as the domain is norm equivalent to a product of Hilbert spaces. We extract pointwise a.e.-converging subsequences and conclude by l.s.c.~of $\mathcal{I}^{f,\theta}$ in $w^{*}$-topology and Fatou's lemma. We obtain l.s.c.~in $w^{*}$-topology as discussed above.\par
Return to the general setting. We show Equation \ref{EQ.LEM.Energy_Functional_I_1}. Let $(\mu,w)\in\Admab$. For all $k\in\mathbb{N}$, definition of quasi-entropy as suprema yields

\begin{align}\label{EQ.LEM.Energy_Functional_I_3}
\mathcal{I}_{k}^{f,\theta}\lc\mu_{k}(t),\mu_{k}(t),w_{k}(t)\rc\leq\mathcal{I}^{f,\theta}\lc\mu(t),\mu(t),w(t)\rc{}    
\end{align}

\noindent for a.e.~$t\in [a,b]$. Note we restrict pointwise. Equation \ref{EQ.LEM.Energy_Functional_I_3} shows $\mathbf{E}^{f,\theta}(\mu,w)\leq E^{f,\theta}(\mu,w)$. Using $1.2)$ in Proposition \ref{PRP.AF_Cstar_Trace_Dualisation_II}, get $w^{*}$-$\lim_{j\in\mathbb{N}}\mu_{j}(t)=\mu(t)$ and $w^{*}$-$\lim_{j\in\mathbb{N}}w_{j}(t)=w(t)$ for a.e.~$t\in [a,b]$. Then l.s.c.~of $\mathcal{I}^{f,\theta}$ and Fatou's lemma imply

\begin{align}\label{EQ.LEM.Energy_Functional_I_4}
E^{f,\theta}(\mu,w)\leq\liminf_{j\in\mathbb{N}}\hspace{0.025cm} \int_{a}^{b}\mathcal{I}_{j}^{f,\theta}\lc\mu_{j}(t),\mu_{j}(t),w_{j}(t)\rc{}dt.
\end{align}

\noindent Yet the right-hand side of Equation \ref{EQ.LEM.Energy_Functional_I_4} equals $\mathbf{E}^{f,\theta}(\mu,w)$ by $3)$ in Theorem \ref{THM.QE_AF}. We altogether obtain our second claim.
\end{proof}

\begin{lem}\label{LEM.Energy_Functional_II}
Let $(\mu^{n},w^{n})_{n\in\mathbb{N}}\subset\Admab$.

\begin{itemize}
\item[1)] If $(\mu,w)=\lim_{n\in\mathbb{N}}\hspace{0.025cm} (\mu^{n},w^{n})_{n\in\mathbb{N}}$ in $\Admab$, then $E^{f,\theta}(\mu,w)\leq\liminf_{n\in\mathbb{N}}E^{f,\theta}(\mu^{n},w^{n})$.

\item[2)] If $\liminf_{n\in\mathbb{N}}E^{f,\theta}(\mu^{n},w^{n})<\infty$ and $t_{0}\in [a,b]$ s.t.~$w^{*}$-$\lim_{n\in\mathbb{N}}\mu^{n}(t_{0})\in\SII(A)$, then there exists a subsequence of $(\mu^{n},w^{n})_{n\in\mathbb{N}}$ converging in canonical topology.
\end{itemize}
\end{lem}
\begin{proof}
By $2)$ in Proposition \ref{PRP.Admissible_Paths_Convergence}, convergence in $\Admab$ implies $w^{*}$-convergence in $L^{2}([a,b],A^{*})_{\w}\times L^{2}([a,b],B^{*})_{\w}$. Thus $1)$ follows from Lemma \ref{LEM.Energy_Functional_I}. We show $2)$. Assume its setting. By passing to subsequences, we furthermore assume $\sup_{n\in\mathbb{N}}E^{f,\theta}(\mu^{n},w^{n})<\infty$ without loss of generality. This is necessary for uniform bounds.\par
Following Remark \ref{REM.Weak_Metrisability}, we metricise the $w^{*}$-topology on $\overline{\SII(A)}$ using $\{x_{n}\}_{n\in\mathbb{N}}\subset A_{0}$ for which the linear span lies $\|.\|_{A}$-dense in $A$ and s.t.~$\dblv{}\nabla x_{n}\dblv_{B}\leq 1$ for all $n\in\mathbb{N}$. Using bounded limit inferior and $2)$ in Proposition \ref{PRP.Energy_Functional_Lipschitz_Sq}, we see $\lset\mu^{n}\rset_{n\in\mathbb{N}}\subset AC([a,b],\overline{\SII(A)})$ is equicontinuous. Note the Arzel\`{a}-Ascoli theorem applies to paths in compact metric spaces \cite{BK.Kel.1975.Topology}. We extract converging sub\-sequence $\lset\mu^{n}\rset_{n\in\mathbb{N}}$. For all $t\in [a,b]$, we obtain $\mu(t):=w^{*}$-$\lim_{n\in\mathbb{N}}\mu^{n}(t)\in\overline{\SII(A)}$. Using $1)$ in Proposition \ref{PRP.Energy_Functional_Lipschitz_Sq} instead, get uniform bound on $\lset{}w^{n}\rset_{n\in\mathbb{N}}\subset L^{2}([a,b],B^{*})_{\w}$. We extract $w^{*}$-converging subsequence $\lset{}w^{n}\rset_{n\in\mathbb{N}}$. Finally, we conclude by applying $1)$ in Proposition \ref{PRP.Admissible_Paths_Convergence} to $(\mu^{n},w^{n})_{n\in\mathbb{N}}$.
\end{proof}

Definition \ref{DFN.Energy_Functional_Restricted} gives suitable restriction of energy functionals. Let $[a,b]\subset\mathbb{R}$. For all $j\in\mathbb{N}$, we know $\resj\circ\incj=\incj$ and $\resjk\circ\inckj=\inckj$ by $2)$ in Proposition \ref{PRP.AF_Cstar_Trace_Dualisation_I}. We therefore identify

\begin{align}\label{EQ.SSEC.QOT_DT_CEN_4}
\Admabj\cong\incj\big(\Admabj\big)\subset\Admab
\end{align}

\noindent in each case. Notation \ref{NTN.AF_Cstar_Trace_Dualisation} thereby likewise extends to admissible paths. For all $j\in\mathbb{N}$ and $(\mu,w)\in\Admabj$, note $1)$ in Proposition \ref{PRP.Energy_Functional_Restriction} shows

\begin{align}\label{EQ.SSEC.QOT_DT_CEN_5}
E_{j}^{f,\theta}\lc\resj(\mu,w)\rc{}=E_{j}^{f,\theta}(\mu,w)=E^{f,\theta}(\mu,w)
\end{align}

\noindent under identification as per Equation \ref{EQ.SSEC.QOT_DT_CEN_4}. Note Equation \ref{EQ.SSEC.QOT_DT_CEN_5} shows Definition \ref{DFN.Energy_Functional_Restricted} extends Equation \ref{EQ.REM.Energy_Functional_1}, i.e.~Definition \ref{DFN.Energy_Functional} for induced noncommutative differential structures. We account for rescaling of norm.

\begin{dfn}\label{DFN.Energy_Functional_Restricted} 
We define the $j$-th restricted energy functional $E^{f,\theta}:\Adm\longrightarrow [0,\infty]$ for $j\in\mathbb{N}$ by setting

\begin{align*}
E_{j}^{f,\theta}(\mu,w):=
\begin{cases}
E_{j}^{f,\theta}\lc\resj(\mu,w)\rc{} & \If\ \mu_{j}\lc{}a\rc\neq 0, \\
0 & \Else.
\end{cases}
\end{align*}
\end{dfn}

\begin{cor}\label{COR.Energy_Functional}
Let $(\mu,w)\in\Admab$. If $E^{f,\theta}(\mu,w)<\infty$, then

\begin{itemize}
\item[1)] $(\mu,w)=\lim_{j\in\mathbb{N}}\resj(\mu,w)$ in $\Admab$,

\item[2)] $E^{f,\theta}(\mu,w)=\lim_{j\in\mathbb{N}}E_{j}^{f,\theta}(\mu,w)$.
\end{itemize}
\end{cor}
\begin{proof}
Let $(\mu,w)\in\Admab$ s.t.~$E^{f,\theta}(\mu,w)<\infty$. Using $1)$ in Proposition \ref{PRP.Energy_Functional_Lipschitz_Sq} to have uniform bounds, get $1)$ by dominated convergence. The necessary pointwise convergence for $\mu$ and $w$ holds by $1)$ in Proposition \ref{PRP.AF_Cstar_Trace_Dualisation_II} as for Remark \ref{REM.AF_Cstar_Trace_Dualisation_Paths}. Lemma \ref{LEM.Energy_Functional_I} and $1)$ in Lemma \ref{LEM.Energy_Functional_II} imply $2)$ since rescaling of norm vanishes in the limit.
\end{proof}

Upon restricting domains to sets of admissible paths with identical interval and marginals, Theorem \ref{THM.Energy_Functional_Representation} shows energy functionals are $\Gamma$-limits of restrictions as per Definition \ref{DFN.Energy_Functional_Restricted}. Sequential descriptions of $\Gamma$-limits require first countability of the domain \cite{BK.DalMas.1993.Gamma_Convergence}. This does not hold for function spaces parametrised by intervals with more than one point \lc{}cf.~Corollary 1.5 in \cite{ART.McCoy.1980.Countability_Function_Spaces}\rc{}. We fix marginals up to restriction in order to get sequential descriptions as per Definition \ref{DFN.Energy_Functional_Lower_Upper_Limits} for non-trivial intervals.

\begin{dfn}\label{DFN.Energy_Functional_Lower_Upper_Limits}
Let $[a,b]\subset\mathbb{R}$.

\begin{itemize}
\item[1)] For all $(\mu,w)\in\Admab$, let $\mathcal{C}(\mu,w)$ be the set of all $\lc\mu^{j},w^{j}\rc_{j\in\mathbb{N}}\subset\Admab$ s.t.~

\begin{itemize}
\item[1.1)] $\lc\mu^{j},w^{j}\rc\in\Admabj\lc\bar{\mu}\lc{}a\rc_{j},\bar{\mu}\lc{}b\rc_{j}\rc$ for a.e.~$j\in\mathbb{N}$, \phantom{\big)}

\item[1.2)] $(\mu,w)=\lim_{j\in\mathbb{N}}\hspace{0.025cm}\lc\mu^{j},w^{j}\rc$ in $\Admab$. \phantom{\big)}
\end{itemize}

\begin{reapply}
\end{reapply}

\item[2)] We define the restricted lower $\Gamma$-limit, resp.~the restricted upper $\Gamma$-limit of $E^{f,\theta}$ by setting

\begin{itemize}
\item[2.1)] $E_{L}^{f,\theta}(\mu,w):=\inf_{C(\mu,w)}\liminf_{j\in\mathbb{N}}E_{j}^{f,\theta}\lc\mu^{j},w^{j}\rc$, \phantom{\big)}

\item[2.2)] $E_{U}^{f,\theta}(\mu,w):=\inf_{C(\mu,w)}\limsup_{j\in\mathbb{N}}E_{j}^{f,\theta}\lc\mu^{j},w^{j}\rc$ \phantom{\big)}
\end{itemize}

\begin{reapply}
\end{reapply}

\noindent for all $(\mu,w)\in\Admab$.
\end{itemize}
\end{dfn}

\begin{rem}\label{REM.Energy_Functional_Lower_Upper_Limits}
Note $1)$ in Definition \ref{DFN.Energy_Functional_Lower_Upper_Limits} equivalently uses truncation $j\geq m$ for fixed but arbitrary $m\in\mathbb{N}$ rather than a.e.~$j\in\mathbb{N}$. The sets we obtain are identical. We use this in the proof of Theorem \ref{THM.Energy_Functional_Representation} if the norm vanishes for finitely many indices.
\end{rem}

\begin{thm}\label{THM.Energy_Functional_Representation}
Let $(\phi,\bpsi,\gamma,\nabla)$ be noncommutative differential structure for tracial AF-$C^{*}$-algebras $(A,\tau)$ and $(B,\omega)$ in $\lc{}f,\theta\rc$-setting. For all $[a,b]\subset\mathbb{R}$ and $\mu^{0},\mu^{1}\in\SII(A)$, get

\begin{align}\label{EQ.THM.Energy_Functional_Representation_1}
\restr{0.925}{E^{f,\theta}}{\Admab\lc\mu^{0},\mu^{1}\rc{}}=\Gamma\textrm{-}\lim_{j\in\mathbb{N}}\hspace{0.05cm} \restr{0.925}{E_{j}^{f,\theta}}{\Admab\lc\mu^{0},\mu^{1}\rc{}}
\end{align}

\noindent and

\begin{align}\label{EQ.THM.Energy_Functional_Representation_2}
\Gamma\textrm{-}\lim_{j\in\mathbb{N}}\hspace{0.05cm} \restr{0.925}{E_{j}^{f,\theta}}{\Admab\lc\mu^{0},\mu^{1}\rc{}}=\restr{0.925}{E_{L}^{f,\theta}}{\Admab\lc\mu^{0},\mu^{1}\rc{}}=\restr{0.925}{E_{U}^{f,\theta}}{\Admab\lc\mu^{0},\mu^{1}\rc{}}.    
\end{align}
\end{thm}
\begin{proof}
Let $[a,b]\subset\mathbb{R}$ and $\mu^{0},\mu^{1}\in\SII(A)$. We use the canonical topology on $\Admab$. Set $X:=\Admab\lc\mu^{0},\mu^{1}\rc$. For all $(\mu,w)\in X$, let $\mathcal{N}(\mu,w)\subset X$ be its set of open neighbourhoods in the relative topology given by $X\subset\Admab$. We show Equation \ref{EQ.THM.Energy_Functional_Representation_1} in the first part of our proof using standard bounds for lower and upper $\Gamma$-limits. We show Equation \ref{EQ.THM.Energy_Functional_Representation_2} in the second part of our proof using Lemma \ref{LEM.Energy_Functional_II} and Corollary \ref{COR.Energy_Functional}. We truncate indices $j\geq m$ in $\mathbb{N}$ as per Remark \ref{REM.Energy_Functional_Lower_Upper_Limits} throughout this proof.\par


\newpage


Let $(\mu,w)\in X$. We show the energy functional bounds its upper $\Gamma$-limit, i.e.~

\begin{align}\label{EQ.THM.Energy_Functional_Representation_3}
\sup_{U\in\mathcal{N}(\mu,w)}\hspace{0.025cm} \limsup_{j\in\mathbb{N}}\hspace{0.025cm} \inf_{(\eta,v)\in U}\hspace{0.025cm} E_{j}^{f,\theta}(\eta,v)\leq E^{f,\theta}(\mu,w).
\end{align}

\noindent Lemma \ref{LEM.Energy_Functional_II} shows $\restr{0.925}{E^{f,\theta}}{X}=\restr{0.925}{\mathbf{E}^{f,\theta}}{X}$. We further know $\lim_{j\in\mathbb{N}}\mu^{0}(1_{A_{j}})=\lim_{j\in\mathbb{N}}\dblv{}\mu_{j}^{0}\dblv_{A^{*}}=1$ by $1.1)$ in Proposition \ref{PRP.AF_Cstar_Trace_Dualisation_II}. For all $U\in\mathcal{N}(\mu,w)$, we have $(\mu,w)\in U$ by definition of open neighbourhood and use $2.1)$ in Proposition \ref{PRP.Energy_Functional_Restriction} in order to estimate

\begin{align*}
\limsup_{j\in\mathbb{N}}\hspace{0.025cm} \inf_{(\eta,v)\in U}\hspace{0.025cm} E_{j}^{f,\theta}(\eta,v) & \leq \limsup_{j\in\mathbb{N}}\hspace{0.025cm} E_{j}^{f,\theta}(\mu,w) \phantom{\bigg)} \\
& \leq \limsup_{j\in\mathbb{N}}\hspace{0.025cm} \mu^{0}(1_{A_{j}})^{-1}\cdot E^{f,\theta}(\mu,w) \phantom{\bigg)} \\
& = E^{f,\theta}(\mu,w). \phantom{\bigg)}
\end{align*}

\noindent Equation \ref{EQ.THM.Energy_Functional_Representation_3} follows by applying the supremum in $\mathcal{N}(\mu,w)$.\par
We show the energy functional is bounded by its lower $\Gamma$-limit, i.e.~

\begin{align}\label{EQ.THM.Energy_Functional_Representation_4}
E^{f,\theta}(\mu,w)\leq\sup_{U\in\mathcal{N}(\mu,w)}\hspace{0.025cm} \liminf_{j\in\mathbb{N}}\hspace{0.025cm} \inf_{(\eta,v)\in U}\hspace{0.025cm} E_{j}^{f,\theta}(\eta,v).
\end{align}

\noindent For all $U\in\mathcal{N}(\mu,w)$, the right-hand side of Equation \ref{EQ.THM.Energy_Functional_Representation_4} is either finite or our claim holds. We assume finiteness without loss of generality. Let $U\in\mathcal{N}(\mu,w)$. We construct a sequence associated to each such open set. For a.e.~$j\in\mathbb{N}$, we have

\begin{align}\label{EQ.THM.Energy_Functional_Representation_5}
E\lc{}U,j\rc{}:=\inf_{(\eta,v)\in U}\hspace{0.025cm} E_{j}^{f,\theta}(\eta,v)<\infty
\end{align}

\noindent by finiteness. We consider subsequence $\lset{}E\lc{}U,j_{n}\rc\rset_{n\in\mathbb{N}}\subset [0,\infty)$ s.t.

\begin{align}\label{EQ.THM.Energy_Functional_Representation_6}
\lim_{n\in\mathbb{N}}\hspace{0.025cm} E\lc{}U,j_{n}\rc{}=\liminf_{j\in\mathbb{N}}\hspace{0.025cm} \inf_{(\eta,v)\in U}\hspace{0.025cm} E_{j}^{f,\theta}(\eta,v)<\infty.
\end{align}

\noindent For a.e.~$j\in\mathbb{N}$, select $\lc\mu^{j},w^{j}\rc\in U$ s.t.~

\begin{align}\label{EQ.THM.Energy_Functional_Representation_7}
\inf_{(\eta,v)\in U}\hspace{0.025cm} E_{j}^{f,\theta}(\eta,v)=E_{j}^{f,\theta}\lc\mu^{j},w^{j}\rc{}+j^{-1}.
\end{align}

\noindent Since marginals are fixed, we use $2)$ in Lemma \ref{LEM.Energy_Functional_II} for $t_{0}=0$ to get subsequence of $\textrm{res}_{jn}\lc\mu^{j_{n}},w^{j_{n}}\rc_{n\in\mathbb{N}}\subset X$ converging in $\Admab$. Note convergence in Equation \ref{EQ.THM.Energy_Functional_Representation_6} is invariant under passing to a subsequence. We relabel the subsequence obtained by Lemma \ref{LEM.Energy_Functional_II} as $\textrm{res}_{jn}\lc\mu^{j_{n}},w^{j_{n}}\rc_{n\in\mathbb{N}}$. Let $\lc\mu^{U},w^{U}\rc$ be its limit in canonical topology.\par


\pagebreak


Using the respective sequence constructed as above in each case, i.e.~s.t.~we have $\lc\mu^{U},w^{U}\rc{}=\lim_{n\in\mathbb{N}}\textrm{res}_{jn}\lc\mu^{j_{n}},w^{j_{n}}\rc$ in $\Admab$, Equation \ref{EQ.THM.Energy_Functional_Representation_5} and Equation \ref{EQ.THM.Energy_Functional_Representation_7} show

\begin{align}\label{EQ.THM.Energy_Functional_Representation_8}
\lim_{n\in\mathbb{N}}\hspace{0.025cm} \bbabsv{1}{\hspace{0.05cm} E_{j_{n}}^{f,\theta}\lc\mu^{j_{n}},w^{j_{n}}\rc{}-E\lc{}U,j_{n}\rc{}\hspace{0.035cm} }=\lim_{n\in\mathbb{N}}\hspace{0.025cm} j_{n}^{-1}=0
\end{align}

\noindent for all $U\in\mathcal{N}(\mu,w)$. Using $1)$ in Lemma \ref{LEM.Energy_Functional_II}, Equation \ref{EQ.THM.Energy_Functional_Representation_6} and Equation \ref{EQ.THM.Energy_Functional_Representation_8} in turn let us calculate

\begin{align*}
E^{f,\theta}\lc\mu^{U},w^{U}\rc{} & \leq \liminf_{n\in\mathbb{N}}\hspace{0.025cm} E_{j_{n}}^{f,\theta}\lc\mu^{j_{n}},w^{j_{n}}\rc \phantom{\bigg)} \\
& = \lim_{n\in\mathbb{N}}\hspace{0.025cm} E\lc{}U,j_{n}\rc \phantom{\bigg)} \\
& = \liminf_{j\in\mathbb{N}}\hspace{0.025cm} \inf_{(\eta,v)\in U}\hspace{0.025cm} E_{j}^{f,\theta}(\eta,v) \phantom{\bigg)}
\end{align*}

\noindent in each case. Equation \ref{EQ.THM.Energy_Functional_Representation_4} therefore follows if

\begin{align}\label{EQ.THM.Energy_Functional_Representation_9}
E^{f,\theta}(\mu,w)\leq\sup_{U\in\mathcal{N}(\mu,w)}\hspace{0.025cm} E^{f,\theta}\lc\mu^{U},w^{U}\rc{}.
\end{align}

\noindent For all $U\in\mathcal{N}(\mu,w)$, we have $\lc\mu^{U},w^{U}\rc\in\overline{U}$ by construction. We thus show Equation \ref{EQ.THM.Energy_Functional_Representation_9} by constructing a sequence of open sets s.t.~Lemma \ref{LEM.Energy_Functional_II} lets us extract subsequence converging to $(\mu,w)$ in $\Admab$ and apply l.s.c.~of the energy functional.\par
Let $K\subset L^{2}([a,b],B^{*})_{\w}$ be a norm bounded closed set s.t.~$w\in K$. We consider $\SII(A)$ and $K$ as metric spaces using $w^{*}$-topology as per Remark \ref{REM.Weak_Metrisability}. All open balls used in this proof are in one of these two metric spaces. Let $\lset{}t_{n}\rset_{n\in\mathbb{N}}\subset [a,b]$ be a dense and monotone increasing sequence. For all $n\in\mathbb{N}$ and $t\in [a,b]$, we define open $V_{n,t}\subset\SII(A)$ by setting

\begin{align*}
V_{n,t}:=
\begin{cases}
B_{n^{-1}}\lc\mu(t_{l})\rc{} & \If\ t=t_{l}\ \textrm{for}\ l\leq n, \\
\SII(A) & \Else.
\end{cases}
\end{align*}

\noindent For all $n\in\mathbb{N}$, set $V_{n}:=\prod_{t\in [a,b]}V_{n,t}$. Each of the latter is an open set in

\begin{align}\label{EQ.THM.Energy_Functional_Representation_10}
\SII(A)^{[a,b]}:=\prod_{t\in [a,b]}\SII(A).
\end{align}

\noindent There exists open sets $\lset{}W_{n}\rset_{n\in\mathbb{N}}\subset\PII\lc{}L^{2}([a,b],B^{*})_{\w}\rc$, the latter denoting the respective power set, s.t.~for all $n\in\mathbb{N}$, we have $W_{n}\cap K=B_{n^{-1}}(w)$ and $W_{n+1}\subset W_{n}$. For all $n\in\mathbb{N}$, we obtain open set $V_{n}\times W_{n}\subset\SII(A)^{[a,b]}\times L^{2}([a,b],B^{*})_{\w}$ and therefore open set

\begin{align}\label{EQ.THM.Energy_Functional_Representation_11}
U_{n}:=\big(V_{n}\times W_{n}\big)\cap X\subset X.
\end{align}

The above construction further ensures

\begin{align}\label{EQ.THM.Energy_Functional_Representation_12}
U_{n+1}\subset U_{n}
\end{align}

\noindent for all $n\in\mathbb{N}$, as well as

\begin{align}\label{EQ.THM.Energy_Functional_Representation_13}
\big\{\hspace{0.025cm} (\mu,w)\hspace{0.025cm} \big\}=\bigcap_{n\in\mathbb{N}}\overline{U}_{n}.
\end{align}

\noindent For all $n\in\mathbb{N}$, let $(\mu^{n},w^{n})\in\overline{U}_{n}$. Using $2)$ in Lemma \ref{LEM.Energy_Functional_II} for $t_{0}=0$, get subsequence converging to $(\mu,w)$ in $\Admab$ by Equation \ref{EQ.THM.Energy_Functional_Representation_12} and Equation \ref{EQ.THM.Energy_Functional_Representation_13}. Equation \ref{EQ.THM.Energy_Functional_Representation_9} follows by applying $1)$ in Lemma \ref{LEM.Energy_Functional_II} to such a subsequence. Equation \ref{EQ.THM.Energy_Functional_Representation_4} follows as discussed above. Altogether, we have Equation \ref{EQ.THM.Energy_Functional_Representation_3} and Equation \ref{EQ.THM.Energy_Functional_Representation_4}. Using standard arguments for $\Gamma$-convergence from upper and lower $\Gamma$-limits \cite{BK.DalMas.1993.Gamma_Convergence}, we see Equation \ref{EQ.THM.Energy_Functional_Representation_3} and Equation \ref{EQ.THM.Energy_Functional_Representation_4} show Equation \ref{EQ.THM.Energy_Functional_Representation_1} immediately.\par
We have $E_{L}^{f,\theta}\leq E_{U}^{f,\theta}$ by definition. We are left to show

\begin{align}\label{EQ.THM.Energy_Functional_Representation_14}
\restr{0.925}{E_{U}^{f,\theta}}{X}\leq\restr{0.925}{E^{f,\theta}}{X}\leq\restr{0.925}{E_{L}^{f,\theta}}{X}.
\end{align}

\noindent Using Equation \ref{EQ.SSEC.QOT_DT_CEN_5} and $1)$ in Lemma \ref{LEM.Energy_Functional_II}, we directly verify 

\begin{align}\label{EQ.THM.Energy_Functional_Representation_15}
\restr{0.925}{E^{f,\theta}}{X}\leq\restr{0.925}{E_{L}^{f,\theta}}{X}.
\end{align}

\noindent Equation \ref{EQ.THM.Energy_Functional_Representation_15} reduces us to $(\mu,w)\in X$ s.t.~$E^{f,\theta}(\mu,w)<\infty$. We assume the latter without loss of generality. Thus $\lset\resj(\mu,w)\rset_{j\in\mathbb{N}}\in\mathcal{C}(\mu,w)$ by $1)$ in Corollary \ref{COR.Energy_Functional}, hence

\begin{align}\label{EQ.THM.Energy_Functional_Representation_16}
\restr{0.925}{E_{U}^{f,\theta}}{X}\leq \restr{0.925}{E^{f,\theta}}{X}
\end{align}

\noindent by $2)$ in Corollary \ref{COR.Energy_Functional}. Equation \ref{EQ.THM.Energy_Functional_Representation_15} and Equation \ref{EQ.THM.Energy_Functional_Representation_16} show Equation \ref{EQ.THM.Energy_Functional_Representation_14}.
\end{proof}


\subsection{Quantum optimal transport distances}\label{SSEC.QOT_DT_MG}

We define quantum optimal transport distances. Theorem \ref{THM.QOT_Distance} collects properties of their metric geometries. Accessibility components are complete geodesic length-metric spaces. Theorem \ref{THM.QOT_Minimiser_Approximation} gives existence of sufficient minimising geodesics approximated in finite dimensions. Standard references for metric geometry are \cite{BK.Amb_Gig_Sav.2008.Classical_OT_GradFlow} and \cite{BK.Bur_Bur_Iv.2001.Metric_Geometry}.


\subsubsection*{Quantum optimal transport as dynamic transport distance}

Let $(\phi,\bpsi,\gamma,\nabla)$ be noncommutative differential structure for tracial AF-$C^{*}$-algebras $(A,\tau)$ and $(B,\omega)$ in $\lc{}f,\theta\rc$-setting. Definition \ref{DFN.QOT_Distance} gives quantum optimal transport distances. It extends the tracial finite-dimensional cases in \cite{ART.Car_Maa.2014.Quantum_OT_I}\cite{ART.Car_Maa.2017.Quantum_OT_II}\cite{ART.Car_Maa.2020.Quantum_OT_III} by construction.

\begin{ntn}
Let $X$ be a set and $d:X\times X\longrightarrow [0,\infty]$ a metric, or distance function on $X$. We say that the metric space $\lc{}X,d\rc$ is equipped with $d$-topology. For all subsets $Y\subset X$, we write $\lc{}Y,d\rc{}=\lc{}Y,d\vert_{Y\times Y}\rc$ for its relative metric space.
\end{ntn}

\begin{dfn}\label{DFN.QOT_Distance}
We define the quantum optimal transport distance of $(\phi,\bpsi,\gamma,\nabla)$ on $\SII(A)$ in $\lc{}f,\theta\rc$-setting by setting

\begin{align}\label{EQ.DFN.QOT_Distance_1}
\mathcal{W}_{\nabla}^{f,\theta}\lc\mu^{0},\mu^{1}\rc{}:=\inf_{\Admnullone(\mu^{0},\mu^{1})}\hspace{0.025cm} \sqrt{E^{f,\theta}(\mu,w)}\in [0,\infty]
\end{align}

\noindent for all $\mu^{0},\mu^{1}\in\SII(A)$.
\end{dfn}

\begin{rem}\label{REM.QOT_Distance}
Neither symmetry of $f$ nor $\|\omega\|^{1-\theta}<\infty$ is required to define admissible paths and energy functionals. They do ensure accessibility components are complete geodesic length-metric spaces. In the logarithmic mean setting, i.e.~$f$ represents the logarithmic operator mean and $\theta=1$, we have symmetric $f$ and $\|\omega\|^{0}=1$.
\end{rem}

We require accessibility components of quantum optimal transport distances to be complete geodesic length-metric spaces. Definition \ref{DFN.Length_Functional} describes length functionals given by integrating square roots of quasi-entropies, i.e.~speed, evaluated on admissible paths. Proposition \ref{PRP.Length_Functional_Structure} shows using square roots of quasi-entropies as speed defines length structures for state spaces in $w^{*}$-topology. Corollary \ref{COR.Length_Functional_Reparametrisation}, which uses constant speed parametrisations of admissible paths on the unit interval as per Lemma \ref{LEM.Length_Functional_Reparametrisation}, in turn shows quantum optimal transport distances are intrinsic distances of such length structures by Proposition 2.4.1 in \cite{BK.Bur_Bur_Iv.2001.Metric_Geometry}. Equation \ref{EQ.COR.Length_Functional_Reparametrisation_1} gives their necessary standard representation. Using our subsequent discussion, Corollary \ref{COR.QOT_Distance_AC_II} shows accessibility components are complete geodesic length-metric spaces.

\begin{dfn}\label{DFN.Length_Functional}\hspace{1cm}
\begin{itemize}
\item[1)] For all $[a,b]\subset\mathbb{R}$ and $(\mu,w)\in\Admab$, set 

\begin{align}\label{EQ.DFN.Length_Functional_1}
\mathcal{N}^{f,\theta}\lc\mu(t),w(t)\rc{}:=\sqrt{\mathcal{I}^{f,\theta}\lc\mu(t),\mu(t),w(t)\rc{}}    
\end{align}

\begin{reapply}
\end{reapply}

\noindent for a.e.~$t\in [a,b]$.

\item[2)] We define the length functional $L^{f,\theta}:\Adm\longrightarrow [0,\infty]$ by setting

\begin{align}\label{EQ.DFN.Length_Functional_2}
L^{f,\theta}(\mu,w):=\int_{a}^{b}\mathcal{N}^{f,\theta}\lc\mu(t),w(t)\rc{}dt
\end{align}

\begin{reapply}
\end{reapply}

\noindent for all $[a,b]\subset\mathbb{R}$ and $(\mu,w)\in\Admab$.
\end{itemize}
\end{dfn}

We restrict admissible paths and therefore length functionals to subintervals as per Remark \ref{REM.Length_Functional}. Proposition \ref{PRP.Length_Functional_Reparametrisation} shows length functionals are invariant under change of variables. Proposition \ref{PRP.Length_Functional_Lipschitz} derives Lipschitz continuity, as well as standard upper bounds involving energy functionals.

\begin{rem}\label{REM.Length_Functional}
We restrict admissible paths to subintervals. Equation \ref{EQ.REM.Length_Functional_1} restricts their length accordingly. Let $[a,b]\subset\mathbb{R}$. For all $[s,t]\subset [a,b]$ and $(\mu,w)\in\Admab$, we have $(\mu,w)\vert_{[s,t]}:=\lc\mu\vert_{[s,t]},w\vert_{[s,t]}\rc\in\Adm^{[s,t]}$ and set

\begin{align}\label{EQ.REM.Length_Functional_1}
\restr{0.925}{L^{f,\theta}(\mu,w)}{[s,t]}:=L^{f,\theta}\lc\mu\vert_{[s,t]},w\vert_{[s,t]}\rc{}=\int_{s}^{t}\mathcal{N}^{f,\theta}\lc\mu(r),w(r)\rc{}dr.
\end{align}
\end{rem}

\begin{prp}\label{PRP.Length_Functional_Reparametrisation}
Let $\varphi:[c,d]\longrightarrow [a,b]$ be monotone and absolutely continuous. If $(\mu,w)\in\Admab\lc\mu^{0},\mu^{1}\rc$, then $\lc\mu\circ\varphi,\dot{\varphi}\cdot \lc{}w\circ\varphi\rc\rc\in\Admcd\lc\mu^{0},\mu^{1}\rc$ and we have

\begin{align}\label{EQ.PRP.Length_Functional_Reparametrisation_1}
L^{f,\theta}(\mu,w)=L^{f,\theta}\lc\mu\circ\varphi,\dot{\varphi}\cdot \lc{}w\circ\varphi\rc\rc{}.
\end{align}
\end{prp}
\begin{proof}
We argue as in the proof of Proposition \ref{PRP.Energy_Functional_Reparametrisation}. However, we integrate over the evaluated square root $\mathcal{N}^{f,\theta}=\sqrt{\mathcal{I}^{f,\theta}}$. Thus we do not require $\dot{\varphi}$ to have $t$-a.e.~defined inverse, hence Equation \ref{EQ.REM.Change_Of_Variables_1} shows Equation \ref{EQ.PRP.Length_Functional_Reparametrisation_1} immediately.
\end{proof}

\begin{prp}\label{PRP.Length_Functional_Lipschitz}
Let $[a,b]\subset\mathbb{R}$.

\begin{itemize}
\item[1)] For all $(\mu,w)\in\Admab$, $x\in A_{0}$ and $[s,t]\subset [a,b]$, we have

\begin{align}\label{EQ.PRP.Length_Functional_Lipschitz_1}
\babsv{1}{\lc\mu(t)-\mu(s)\rc{}(x)}\leq\restr{0.925}{L^{f,\theta}(\mu,w)}{[s,t]}\cdot \sqrt{2^{-\theta}\lc\|\phi\|_{1}^{\theta}+\|\bpsi\|_{1}^{\theta}\rc\cdot \|\omega\|^{1-\theta}}\cdot \|\nabla x\|_{B}.
\end{align}

\begin{reapply}
\end{reapply}

\item[2)] For all $(\mu,w)\in\Admab$, we have

\begin{align}\label{EQ.PRP.Length_Functional_Lipschitz_2}
L^{f,\theta}(\mu,w)^{2}\leq (b-a)\cdot E^{f,\theta}(\mu,w).
\end{align}

\begin{reapply}
\end{reapply}

\noindent Furthermore, we have equality in Equation \ref{EQ.PRP.Length_Functional_Lipschitz_1} if and only if $t\mapsto\mathcal{N}^{f,\theta}\lc\mu(t),w(t)\rc$ is $t$-a.e.~constant on $[a,b]$.
\end{itemize}
\end{prp}
\begin{proof}
We show $1)$. We argue as in the proof of $2)$ in Proposition \ref{PRP.Energy_Functional_Lipschitz_Sq}, where we use the continuity equation to estimate. Rather than subsequent application of H\"older, we instead apply $5)$ in Theorem \ref{THM.QE_AF}. Equation \ref{EQ.PRP.Length_Functional_Lipschitz_1} holds. We show $2)$. We reduce to $[a,b]=[0,1]$ by applying Proposition \ref{PRP.Length_Functional_Reparametrisation} to the left-~, resp.~Proposition \ref{PRP.Energy_Functional_Reparametrisation} to the right-hand side of Equation \ref{EQ.PRP.Length_Functional_Lipschitz_2}. We use affine transformations as per Remark \ref{REM.Energy_Functional_Reparametrisation} in both cases. Having reduced to $[0,1]$ as described, both Equation \ref{EQ.PRP.Length_Functional_Lipschitz_2} and our claim concerning equality follow by Jensen's inequality.
\end{proof}

\begin{prp}\label{PRP.Length_Functional_Structure}
$\lc\Adm,L^{f,\theta}\rc$ is a length structure for $\SII(A)$ in $w^{*}$-topology.
\end{prp}
\begin{proof}
Proposition \ref{PRP.Length_Functional_Reparametrisation} shows $\Adm$ is a class of admissible paths in the sense of metric geometry \cite{BK.Bur_Bur_Iv.2001.Metric_Geometry}. Our claim follows if $L^{f,\theta}$ satisfies conditions $1)$ to $4)$ on p.27 in \cite{BK.Bur_Bur_Iv.2001.Metric_Geometry}. Using Equation \ref{EQ.DFN.Length_Functional_2}, we directly verify the first three conditions. The fourth one is equivalent to the following statement. If $\lset\mu^{n}\rset_{n\in\mathbb{N}}\subset\SII(A)$ and $\mu^{0}\in\SII(A)$ s.t.~

\begin{align}\label{EQ.PRP.Length_Functional_Structure_1}
\lim_{n\in\mathbb{N}}\hspace{0.025cm} \inf_{\Adm(\mu^{0},\mu^{n})}\hspace{0.025cm} L^{f,\theta}(\mu,w)=0,
\end{align}

\noindent then $\mu=w^{*}$-$\lim\mu^{n}$ in $\SII(A)$. This is ensured by $1)$ in Proposition \ref{PRP.Length_Functional_Lipschitz}.
\end{proof}

\begin{lem}\label{LEM.Length_Functional_Reparametrisation}
Let $\mu^{0},\mu^{1}\in\SII(A)$ and $(\mu,w)\in\Admnullone\lc\mu^{0},\mu^{1}\rc$. If $L^{f,\theta}(\mu,w)\in (0,\infty)$, then there exists $(\tilde{\mu},\tilde{w})\in\Admnullone\lc\mu^{0},\mu^{1}\rc$ s.t.~

\begin{itemize}
\item[1)] $L^{f,\theta}(\tilde{\mu},\tilde{w})=L^{f,\theta}(\mu,w)$,

\item[2)] $t\mapsto\mathcal{N}^{f,\theta}\lc\tilde{\mu}(t),\tilde{w}(t)\rc\neq 0$ is $t$-a.e.~constant.
\end{itemize}
\end{lem}
\begin{proof}
Assume $L^{f,\theta}(\mu,w)\in (0,\infty)$. For all $t\in [0,1]$, set

\begin{align}\label{EQ.LEM.Length_Functional_Reparametrisation_1}
\varphi(t):=\varphi^{\mu,w}(t):=L^{f,\theta}(\mu,w)^{-1}\cdot \int_{0}^{t}\mathcal{N}^{f,\theta}\lc\mu(s),w(s)\rc{}ds.
\end{align}

\noindent Since $L^{f,\theta}(\mu,w)>0$, get $\mu^{0}\neq\mu^{1}$ by Proposition \ref{PRP.Length_Functional_Structure}. Since $L^{f,\theta}(\mu,w)<\infty$, we know $\varphi$ is monotone and absolutely continuous. We reduce to $\varphi$ strictly monotone.\par
Assume $\varphi$ is not strictly monotone. There exists $[c,d]\subset [0,1]$ s.t.~$t\mapsto\mathcal{N}^{f,\theta}\lc\mu(t),w(t)\rc$ vanishes for a.e.~$t\in [c,d]$. Thus $\mu\vert_{[c,d]}$ is constant by $1)$ in Proposition \ref{PRP.Length_Functional_Lipschitz}. We select $[c,d]$ maximal. If $[c,d]\subset I$ proper for a closed interval $I\subset\mathbb{R}$, then $\mu\vert_{I}$ is not constant on $I$. Since $\mu^{0}\neq\mu^{1}$ and $[0,1]$ is compact, there exists $m\in\mathbb{N}$ and non-intersecting maximal intervals $\lset{}[c_{n},d_{n}]\rset_{n=1}^{m}\subset\PII(\mathbb{R})$ satisfying $\textrm{R}.1\rc$ and $\textrm{R}.2\rc$ below. For all $n\in\lset{}1,\ldots,m\rset$, let

\begin{itemize}
\item[R.1)] $0<c_{n}<d_{n}<1$ s.t.~$\mu\vert_{[c_{n},d_{n}]}$ is constant,

\item[R.2)] there exists no $(a,b)\subset [0,1]\mathbin{\big\backslash}\lc\bigcup_{n=1}^{m}[c_{n},d_{n}]\rc$ s.t.~$\mu\vert_{[a,b]}$ is constant.
\end{itemize}

Set $d_{0}:=0$ and $c_{m+1}:=1$. For all $n\in\lset{}0,\ldots,m\rset$, get $(\mu,w)\vert_{[d_{n},c_{n+1}]}\in\Adm$. Thus $\textrm{R}.1\rc$ immediately yields

\begin{align}\label{EQ.LEM.Length_Functional_Reparametrisation_2}
L^{f,\theta}(\mu,w)=\sum_{n=0}^{m}\restr{0.925}{L^{f,\theta}(\mu,w)}{[d_{n},c_{n+1}]}.
\end{align}

\noindent For all $n\in\lset{}0,\ldots,m\rset$, we reparametrise $(\mu,w)\vert_{[d_{n},c_{n+1}]}$ to $[n\lc{}m+1\rc^{-1},\lc{}n+1\rc\lc{}m+1\rc^{-1}]$ using affine transformation as per Remark \ref{REM.Energy_Functional_Reparametrisation}. We concatenate reparametrised paths via canonical topological path composition.\par


\pagebreak


Altogether, we obtain a rectified path

\begin{align}\label{EQ.LEM.Length_Functional_Reparametrisation_3}
(\tilde{\mu},\tilde{w}):=\restr{0.925}{(\mu,w)}{\big[0,\frac{1}{m+1}\big]}\circ\cdots \circ\restr{0.925}{(\mu,w)}{\big[\frac{m}{m+1},1\big]}\in\Admnullone.
\end{align}

\noindent Proposition \ref{PRP.Length_Functional_Reparametrisation} and Equation \ref{EQ.LEM.Length_Functional_Reparametrisation_2} show $L^{f,\theta}(\tilde{\mu},\tilde{w})=L^{f,\theta}(\mu,w)$. Yet $\textrm{R}.2\rc$ shows there exists no $(a,b)\subset [0,1]$ s.t.~$\tilde{\mu}\vert_{[a,b]}$ is constant. Hence $\varphi^{\tilde{\mu},\tilde{w}}$ is strictly monotone by $1)$ in Proposition \ref{PRP.Length_Functional_Lipschitz}. We may reduce since its construction preserves length.\par
We assume $\varphi:=\varphi^{\mu,w}$ is strictly monotone without loss of generality. Thus $\varphi$ is a homeomorphism onto $[0,1]$, hence $\varphi^{-1}$ exists and is monotone. Monotonicity ensures $\varphi^{-1}$ has $t$-a.e.~finite derivative $\frac{d}{dt}\varphi^{-1}$. The chain rule holds for $\mu\circ\varphi^{-1}$ upon testing with $A_{0}$ \lc{}cf.~Corollary 4 in \cite{ART.Ser_Var.1969.General_Chain_Rule}\rc{}. We therefore have

\begin{align}\label{EQ.LEM.Length_Functional_Reparametrisation_4}
(\tilde{\mu},\tilde{w}):=\bigg(\mu\circ\varphi^{-1},\frac{d}{dt}\varphi^{-1}\cdot \lc{}w\circ\varphi^{-1}\rc\bigg)\in\Admnullone\lc\mu^{0},\mu^{1}\rc{}.
\end{align}

\noindent Proposition \ref{PRP.Length_Functional_Reparametrisation} shows $L^{f,\theta}(\tilde{\mu},\tilde{w})=L^{f,\theta}(\mu,w)$. Since we have $t$-a.e.~finite derivatives for $\varphi$, $\varphi^{-1}$ and $\id_{[0,1]}$, the chain rule holds for $t=\varphi\lc\varphi^{-1}(t)\rc$ \lc{}cf.~Theorem 2 in \cite{ART.Ser_Var.1969.General_Chain_Rule}\rc{}. We use chain rule to derive the first, and Equation \ref{EQ.LEM.Length_Functional_Reparametrisation_1} for the second identity in

\begin{align}\label{EQ.LEM.Length_Functional_Reparametrisation_5}
\restr{0.925}{\frac{d}{ds}}{s=t}\varphi^{-1}(s)=\lc\restr{0.925}{\frac{d}{ds}}{s=\varphi^{-1}(t)}\varphi(s)\rc^{-1}=L^{f,\theta}(\mu,w)\cdot \mathcal{N}^{f,\theta}\lc\mu\lc\varphi^{-1}(t)\rc{},w\lc\varphi^{-1}(t)\rc\rc^{-1}
\end{align}

\noindent for a.e.~$t\in [0,1]$. Using Equation \ref{EQ.LEM.Length_Functional_Reparametrisation_5}, we further calculate

\begin{align}\label{EQ.LEM.Length_Functional_Reparametrisation_6}
\mathcal{N}^{f,\theta}\lc\tilde{\mu}(t),\tilde{w}(t)\rc{}=\restr{0.925}{\frac{d}{ds}}{s=t}\varphi^{-1}(s)\cdot \mathcal{N}^{f,\theta}\lc\mu\lc\varphi^{-1}(t)\rc{},w\lc\varphi^{-1}(t)\rc\rc{}=L^{f,\theta}(\mu,w)\neq 0
\end{align}

\noindent in each case. Equation \ref{EQ.LEM.Length_Functional_Reparametrisation_6} shows our claim.
\end{proof} 

\begin{rem}
In the proof of Lemma \ref{LEM.Length_Functional_Reparametrisation}, we alternatively show $\tilde{\mu}$ has constant and non-vanishing metric derivative. Minimality and finite length let us bound from below using the metric derivative in order to show $\mathcal{N}^{f,\theta}\lc\tilde{\mu}(t),\tilde{w}(t)\rc\neq 0$ for a.e.~$t\in [0,1]$.
\end{rem}

\begin{cor}\label{COR.Length_Functional_Reparametrisation}
For all $\mu^{0},\mu^{1}\in\SII(A)$, we have 

\begin{align}\label{EQ.COR.Length_Functional_Reparametrisation_1}
\mathcal{W}_{\nabla}^{f,\theta}\lc\mu^{0},\mu^{1}\rc{}=\inf_{\Admnullone(\mu^{0},\mu^{1})}\hspace{0.025cm} L^{f,\theta}(\mu,w)\ =\inf_{\Adm(\mu^{0},\mu^{1})}\hspace{0.025cm} L^{f,\theta}(\mu,w).
\end{align}
\end{cor}
\begin{proof}
Let $\mu^{0},\mu^{1}\in\SII(A)$. Either $\mu^{0}\neq\mu^{1}$ or all terms equal zero. We assume $\mu^{0}\neq\mu^{1}$ without loss of generality. Proposition \ref{PRP.Length_Functional_Reparametrisation} shows

\begin{align}\label{EQ.COR.Length_Functional_Reparametrisation_2}
\inf_{\Admnullone(\mu^{0},\mu^{1})}\hspace{0.025cm} L^{f,\theta}(\mu,w)\ =\inf_{\Adm(\mu^{0},\mu^{1})}\hspace{0.025cm} L^{f,\theta}(\mu,w).
\end{align}

Moreover, $2)$ in Proposition \ref{PRP.Length_Functional_Lipschitz} shows

\begin{align}\label{EQ.COR.Length_Functional_Reparametrisation_3}
\inf_{\Admnullone(\mu^{0},\mu^{1})}\hspace{0.025cm} L^{f,\theta}(\mu,w)\leq\mathcal{W}_{\nabla}^{f,\theta}\lc\mu^{0},\mu^{1}\rc{}.
\end{align}

\noindent Let $\mathbf{S}:=\lset{}(\mu,w)\in\Admnullone\lc\mu^{0},\mu^{1}\rc\ \vset\ t\mapsto\mathcal{N}^{f,\theta}\lc\mu(t),w(t)\rc\ \textrm{is}\ t\textrm{-a.e.~constant}\rset$. Lemma \ref{LEM.Length_Functional_Reparametrisation} implies $\mathbf{S}\neq\emptyset$ and further

\begin{align}\label{EQ.COR.Length_Functional_Reparametrisation_4}
\inf_{\mathbf{S}}\ \hspace{0.025cm} L^{f,\theta}(\mu,w)\ =\inf_{\Admnullone(\mu^{0},\mu^{1})}\hspace{0.025cm} L^{f,\theta}(\mu,w).
\end{align}

\noindent Using the statement on equality for Equation \ref{EQ.PRP.Length_Functional_Lipschitz_2}, we see Equation \ref{EQ.COR.Length_Functional_Reparametrisation_4} shows

\begin{align}\label{EQ.COR.Length_Functional_Reparametrisation_5}
\mathcal{W}_{\nabla}^{f,\theta}\lc\mu^{0},\mu^{1}\rc\leq\inf_{\mathbf{S}}\ \hspace{0.025cm} L^{f,\theta}(\mu,w)\ =\inf_{\Admnullone(\mu^{0},\mu^{1})}\hspace{0.025cm} L^{f,\theta}(\mu,w).
\end{align}

\noindent Equation \ref{EQ.COR.Length_Functional_Reparametrisation_2}, Equation \ref{EQ.COR.Length_Functional_Reparametrisation_3} and Equation \ref{EQ.COR.Length_Functional_Reparametrisation_5} imply Equation \ref{EQ.COR.Length_Functional_Reparametrisation_1}.
\end{proof}

Definition \ref{DFN.QOT_Distance_Geodesics} gives minimising geodesics and distance minimisers \cite{BK.Amb_Gig_Sav.2008.Classical_OT_GradFlow}\cite{BK.Bur_Bur_Iv.2001.Metric_Geometry}. The notions coincide by $4)$ in Theorem \ref{THM.QOT_Distance}. In Section \ref{SEC.L2W_EVI}, we apply results in variational analysis for metric geometry using minimising geodesics \cite{ART.Dan_Sav.2008.Classical_OT_GradFlow_DisConvex}\cite{ART.Mur_Sav.2020.Classical_OT_EVI}.

\begin{dfn}\label{DFN.QOT_Distance_Geodesics}\hspace{1cm}
\begin{itemize}
\item[1)] Let $\mu^{0},\mu^{1}\in\SII(A)$ and $[a,b]\subset\mathbb{R}$. We call $(\mu,w)\in\Admab\lc\mu^{0},\mu^{1}\rc$ a minimising geodesic from $\mu^{0}$ to $\mu^{1}$ if there exists $C\geq 0$ s.t.~

\begin{align}\label{EQ.DFN.QOT_Distance_Geodesics_1}
\mathcal{W}_{\nabla}^{f,\theta}\lc\mu(t),\mu(s)\rc{}=C\absv{1}{t-s}
\end{align}

\begin{reapply}
\end{reapply}

\noindent for all $t,s\in [a,b]$.

\item[2)] Let $\mu^{0},\mu^{1}\in\SII(A)$. We call $(\mu,w)\in\Admnullone\lc\mu^{0},\mu^{1}\rc$ a distance minimiser if

\begin{align}\label{EQ.DFN.QOT_Distance_Geodesics_2}
\mathcal{W}_{\nabla}^{f,\theta}\lc\mu^{0},\mu^{1}\rc{}=\sqrt{E^{f,\theta}(\mu,w)}<\infty.
\end{align}

\begin{reapply}
\end{reapply}

\item[3)] For all $\mu^{0},\mu^{1}\in\SII(A)$, let $\Geo\lc\mu^{0},\mu^{1}\rc$ be the set of all distance minimisers with marginals $\mu^{0}$ and $\mu^{1}$. Set $\Geo:=\bigcup_{\mu^{0},\mu^{1}\in\SII(A)}\Geo\lc\mu^{0},\mu^{1}\rc$.
\end{itemize}
\end{dfn}

\begin{ntn}\label{NTN.QOT_Distance_Geodesics}
For all $j\in\mathbb{N}$, we use $\Geoj$ when denoting sets of distance minimisers in Definition \ref{DFN.QOT_Distance_Geodesics} for induced noncommutative differential structure $\lc\phi_{j},\bpsi_{j},\gamma_{j},\nabla_{\hspace{-0.055cm} j}\rc$.
\end{ntn}


\pagebreak


\begin{prp}\label{PRP.QOT_Distance_Geodesics}\hspace{1cm}
\begin{itemize}
\item[1)] Let $\mu^{0},\mu^{1}\in\SII(A)$. If $(\mu,w)\in \Geo\lc\mu^{0},\mu^{1}\rc$, then $t\mapsto\mathcal{N}^{f,\theta}\lc\mu(t),w(t)\rc$ is $t$-a.e.~constant.

\item[2)] For all $j\leq k$ in $\mathbb{N}$ and $\mu^{0},\mu^{1}\in\mathcal{S}(A_{j})$, we have length- and energy-preserving maps

\begin{itemize}
\item[2.1)] $\Geoj\lc\mu^{0},\mu^{1}\rc\overset{\inckj}{\lhook\joinrel\longrightarrow}\Geok\lc\mu^{0},\mu^{1}\rc\overset{\incj}{\lhook\joinrel\longrightarrow}\Geo\lc\mu^{0},\mu^{1}\rc$, \phantom{\bigg)}

\item[2.2)] $\Geo\lc\mu^{0},\mu^{1}\rc\overset{\resk}{\longrightarrow}\Geok\lc\mu^{0},\mu^{1}\rc\overset{\resjk}{\longrightarrow}\Geoj\lc\mu^{0},\mu^{1}\rc$. \phantom{\bigg)}
\end{itemize}

\begin{reapply}
\end{reapply}

\end{itemize}
\end{prp}
\begin{proof}
We show $1)$. For all $(\mu,w)\in\Geo\lc\mu^{0},\mu^{1}\rc$, $2)$ in Proposition \ref{PRP.Length_Functional_Lipschitz}, Corollary \ref{COR.Length_Functional_Reparametrisation} and Equation \ref{EQ.DFN.QOT_Distance_Geodesics_2} yield

\begin{align}\label{EQ.PRP.QOT_Distance_Geodesics_1}
\mathcal{W}_{\nabla}^{f,\theta}\lc\mu^{0},\mu^{1}\rc\leq L^{f,\theta}\leq\sqrt{E^{f,\theta}(\mu,w)}=\mathcal{W}_{\nabla}^{f,\theta}\lc\mu^{0},\mu^{1}\rc{}.
\end{align}

\noindent Furthermore, we have equality in Equation \ref{EQ.PRP.QOT_Distance_Geodesics_1} if and only if $t\mapsto\mathcal{N}^{f,\theta}\lc\mu(t),w(t)\rc$ is $t$-a.e.~constant on $[0,1]$. We have $1)$. We show $2)$. For all distance minimisers, equality in Equation \ref{EQ.PRP.QOT_Distance_Geodesics_1} implies length and square root of energy coincide. It suffices to show energy is preserved. Inclusions in $2.1)$ preserve energy by $1)$ in Proposition \ref{PRP.Energy_Functional_Restriction}, and restrictions in $2.2)$ do not increase energy by $2)$ in Proposition \ref{PRP.Energy_Functional_Restriction}. We obtain $2)$ by Equation \ref{EQ.DFN.QOT_Distance_Geodesics_2} since restriction maps are left-inverses of inclusion maps.
\end{proof}

\begin{ntn}
For all $\mu^{k}\in\SII(A)$ and $k\in\lset{}0,1\rset$, set $\bar{\mu}^{k}:=\overline{\mu^{k}}$ as per Definition \ref{DFN.AF_Cstar_Trace_Dualisation_Paths}.
\end{ntn}

\begin{thm}\label{THM.QOT_Distance}
Let $(\phi,\bpsi,\gamma,\nabla)$ be noncommutative differential structure for tracial AF-$C^{*}$-algebras $(A,\tau)$ and $(B,\omega)$ in $\lc{}f,\theta\rc$-setting.

\begin{itemize}
\item[1)] $\big(\hspace{-0.03875cm} \SII(A),\mathcal{W}_{\nabla}^{f,\theta}\big)$ is a length-metric space with topology stronger than $w^{*}$-topology.

\item[2)] For all $j\leq k$ in $\mathbb{N}$, we have isometric inclusions

\begin{align}\label{EQ.THM.QOT_Distance_1}
\big(\mathcal{S}(A_{j}),\mathcal{W}_{\nabla_{\hspace{-0.055cm} j}}^{f,\theta}\big)\overset{\inckj}{\lhook\joinrel\longrightarrow}\big(\mathcal{S}(A_{k}),\mathcal{W}_{\nabla_{\hspace{-0.055cm} k}}^{f,\theta}\big)\overset{\inck}{\lhook\joinrel\longrightarrow}\big(\hspace{-0.03875cm} \SII(A),\mathcal{W}_{\nabla}^{f,\theta}\big){}.
\end{align}

\begin{reapply}
\end{reapply}

\item[3)] $\mathcal{W}_{\nabla}^{f,\theta}$ is l.s.c.~in $w^{*}$-topology. For all $\mu^{0},\mu^{1}\in\SII(A)$, we have

\begin{align}\label{EQ.THM.QOT_Distance_2}
\mathcal{W}_{\nabla}^{f,\theta}\lc\mu^{0},\mu^{1}\rc{}=\lim_{j\in\mathbb{N}}\hspace{0.025cm} \mathcal{W}_{\nabla_{\hspace{-0.055cm} j}}^{f,\theta}\big(\bar{\mu}_{j}^{0},\bar{\mu}_{j}^{1}\big).
\end{align}

\begin{reapply}
\end{reapply}

\item[4)] Let $\mu^{0},\mu^{1}\in\SII(A)$.

\begin{itemize}
\item[4.1)] If $\mathcal{W}_{\nabla}^{f,\theta}\lc\mu^{0},\mu^{1}\rc{}<\infty$, then $\Geo\lc\mu^{0},\mu^{1}\rc\neq\emptyset$. \phantom{\big)}

\item[4.2)] For all $(\mu,w)\in\Admnullone\lc\mu^{0},\mu^{1}\rc$, we have $(\mu,w)\in\Geo\lc\mu^{0},\mu^{1}\rc$ if and only if $\mu$ is a minimising geodesic from $\mu^{0}$ to $\mu^{1}$. \phantom{\big)}
\end{itemize}

\begin{reapply}
\end{reapply}

\end{itemize}
\end{thm}


\pagebreak


\begin{proof}
We know $1)$ by Proposition \ref{PRP.Length_Functional_Structure} and Corollary \ref{COR.Length_Functional_Reparametrisation}. Then $2)$ follows from $2)$ in Proposition \ref{PRP.QOT_Distance_Geodesics}. We show $3)$. Let $\mu^{0},\mu^{1}\in\SII(A)$. For all $k\in\lset{}0,1\rset$, let $\lset\mu^{n,k}\rset_{n\in\mathbb{N}}\subset\SII(A)$ s.t.~$\mu^{k}=w^{*}$-$\lim_{n\in\mathbb{N}}\mu^{n,k}$, $\mathcal{W}_{\nabla}^{f,\theta}\lc\mu^{n,0},\mu^{n,1}\rc{}<\infty$ for all $n\in\mathbb{N}$, as well as

\begin{align}\label{EQ.THM.QOT_Distance_3}
\liminf_{n\in\mathbb{N}}\hspace{0.0675cm} \mathcal{W}_{\nabla}^{f,\theta}\lc\mu^{n,0},\mu^{n,1}\rc{}<\infty.
\end{align}

In order to show l.s.c.~in $w^{*}$-topology, it suffices to consider such subsequences. For all $n\in\mathbb{N}$, let $(\mu^{n},w^{n})\in\Admnullone\lc\mu^{n,0},\mu^{n,1}\rc$ s.t.~

\begin{align}\label{EQ.THM.QOT_Distance_4}
E^{f,\theta}(\mu^{n},w^{n})=\mathcal{W}_{\nabla}^{f,\theta}\lc\mu^{n,0},\mu^{n,1}\rc^{2}+n^{-1}.
\end{align}

\noindent Using $w^{*}$-convergence of marginals, Equation \ref{EQ.THM.QOT_Distance_4} shows Lemma \ref{LEM.Energy_Functional_II} for $t_{0}=0$ yields $(\mu,w)\in\Admnullone\lc\mu^{0},\mu^{1}\rc$ s.t.~we have estimate

\begin{align}\label{EQ.THM.QOT_Distance_5}
\mathcal{W}_{\nabla}^{f,\theta}\lc\mu^{0},\mu^{1}\rc\leq\sqrt{E^{f,\theta}(\mu,w)}\leq\liminf_{n\in\mathbb{N}}\hspace{0.0675cm} \mathcal{W}_{\nabla}^{f,\theta}\lc\mu^{n,0},\mu^{n,1}\rc{}.
\end{align}

\noindent Equation \ref{EQ.THM.QOT_Distance_5} shows l.s.c.~in $w^{*}$-topology. In particular, we see Equation \ref{EQ.THM.QOT_Distance_2} follows at once if

\begin{align}\label{EQ.THM.QOT_Distance_6}
\limsup_{j\in\mathbb{N}}\hspace{0.025cm} \mathcal{W}_{\nabla_{\hspace{-0.055cm} j}}^{f,\theta}\big(\bar{\mu}_{j}^{0},\bar{\mu}_{j}^{1}\big)\leq\mathcal{W}_{\nabla}^{f,\theta}\lc\mu^{0},\mu^{1}\rc{}.
\end{align}

\noindent Equation \ref{EQ.THM.QOT_Distance_6} holds by $2.1)$ in Proposition \ref{PRP.Energy_Functional_Restriction}. Get $3)$.\par
We show $4)$. Lemma \ref{LEM.Energy_Functional_II} for $t_{0}=0$ implies $4.1)$. Let $(\mu,w)\in\Geo\lc\mu^{0},\mu^{1}\rc$. Using Corollary \ref{COR.Length_Functional_Reparametrisation} and $1)$ in Proposition \ref{PRP.QOT_Distance_Geodesics}, get $C:=\mathcal{N}^{f,\theta}\lc\mu(t),w(t)\rc$ for a.e.~$t\in [0,1]$ and estimate

\begin{align}\label{EQ.THM.QOT_Distance_7}
\mathcal{W}_{\nabla}^{f,\theta}\lc\mu(s),\mu(t)\rc\leq\int_{s}^{t}\mathcal{N}^{f,\theta}\lc\mu(r),w(r)\rc{}dr=C\absv{1}{t-s}
\end{align}

\noindent for all $t,s\in [0,1]$. Let $[s,t]\subset [0,1]$ proper. If equality in Equation \ref{EQ.THM.QOT_Distance_6} does not hold for $[s,t]\subset [0,1]$, then there exists a distance minimiser from $\mu(s)$ to $\mu(t)$ with strictly less length than $(\mu,w)\vert_{[s,t]}$. Note Remark \ref{REM.Length_Functional}. This contradicts minimality on $[0,1]$. Thus equality holds in each case, hence $\mu$ is a minimising geodesic. The converse then follows by Equation \ref{EQ.DFN.QOT_Distance_Geodesics_1} and Theorem 2.7.6 in \cite{BK.Bur_Bur_Iv.2001.Metric_Geometry}. Get $4.2)$. Altogether, get $4)$.
\end{proof}


\pagebreak



\subsubsection*{Accessibility components and minimising geodesics}

Definition \ref{DFN.QOT_Distance_AC} gives accessibility components of quantum optimal transport distances. They are maximal sets of states at finite distance. Corollary \ref{COR.QOT_Distance_AC_II} shows accessibility components are complete geodesic length-metric spaces s.t.~intrinsic distances of their length structures are quantum optimal transport distances. Thus accessibility components are maximal sets of points connected by minimising geodesics, hence metric geometry reduces to the latter. We use this throughout our discussion.\par
Let $(\phi,\bpsi,\gamma,\nabla)$ be noncommutative differential structure for tracial AF-$C^{*}$-algebras $(A,\tau)$ and $(B,\omega)$ in $\lc{}f,\theta\rc$-setting.

\begin{dfn}\label{DFN.QOT_Distance_AC}\hspace{1cm}
\begin{itemize}
\item[1)] We call $\mathcal{C}\subset\SII(A)$ accessible if $\mathcal{W}_{\nabla}^{f,\theta}\lc\mu^{0},\mu^{1}\rc{}<\infty$ for all $\mu^{0},\mu^{1}\in\mathcal{C}$.

\item[2)] We say that $\mathcal{C}\subset\SII(A)$ is an accessibility component if there exists no accessible $\mathcal{C}'\subset\SII(A)$ s.t.~$\mathcal{C}'\subset\mathcal{C}$ proper. If $\mathcal{C}\subset\SII(A)$ is an accessibility component, then we write 

\begin{align}\label{EQ.DFN.QOT_Distance_AC_1}
\mathcal{C}\subset\big(\hspace{-0.03875cm} \SII(A),\mathcal{W}_{\nabla}^{f,\theta}\big).    
\end{align}

\begin{reapply}
\end{reapply}

\item[3)] For all $\mathcal{C}\subset\big(\hspace{-0.03875cm} \SII(A),\mathcal{W}_{\nabla}^{f,\theta}\big)$, set $\Adm_{\mathcal{C}}:=\bigcup_{\mu,\eta\in\mathcal{C}}\Adm(\mu,\eta)$.
\end{itemize}
\end{dfn}

\begin{cor}\label{COR.QOT_Distance_AC_I}\hspace{1cm}
\begin{itemize}
\item[1)] An equivalence relation on $\SII(A)$ is given by

\begin{align}\label{EQ.COR.QOT_Distance_AC_I_1}
\mu\sim\eta\Leftrightarrow \mu,\eta\in\mathcal{C}\subset\big(\hspace{-0.03875cm} \SII(A),\mathcal{W}_{\nabla}^{f,\theta}\big)
\end{align}

\begin{reapply}
\end{reapply}

\noindent for all $\mu,\eta\in\SII(A)$.

\item[2)] For all $\mu,\eta\in\SII(A)$, we have $\mu\sim\eta$ if and only if

\begin{itemize}
\item[2.1)] $\bar{\mu}_{j}\sim\bar{\eta}_{j}$ for a.e.~$j\in\mathbb{N}$,

\item[2.2)] $\lim\sup_{j\in\mathbb{N}}\mathcal{W}_{\nabla_{\hspace{-0.055cm} j}}^{f,\theta}\lc\bar{\mu}_{j},\bar{\eta}_{j}\rc{}<\infty$.
\end{itemize}

\begin{reapply}
\end{reapply}

\end{itemize}
\end{cor}
\begin{proof}
Let $\mathcal{C}\subset (\SII(A),\mathcal{W}_{\nabla}^{f,\theta})$. If $\mu^{0}\in\mathcal{C}$, then $1)$ in Theorem \ref{THM.QOT_Distance} and maximality of $\mathcal{C}$ as set of finite-length admissible paths shows

\begin{align}\label{EQ.COR.QOT_Distance_AC_I_2}
\mathcal{C}=\lset\mu^{1}\in\SII(A)\ \vset\ \exists(\mu,w)\in\Admnullone\lc\mu^{0},\mu^{1}\rc{}:\ L^{f,\theta}(\mu,w)<\infty\rset{}.
\end{align}

\noindent Using Equation \ref{EQ.COR.QOT_Distance_AC_I_2}, we directly verify Equation \ref{EQ.COR.QOT_Distance_AC_I_1} defines an equivalence relation on $\SII(A)$. Thus $1)$ holds, hence $2)$ follows from $3)$ in Theorem \ref{THM.QOT_Distance}.
\end{proof}

\begin{cor}\label{COR.QOT_Distance_AC_II}
For all $\mathcal{C}\subset (\SII(A),\mathcal{W}_{\nabla}^{f,\theta})$, we have

\begin{itemize}
\item[1)] $\lc{}L^{f,\theta},\Adm_{\mathcal{C}}\rc$ is a length structure for $\mathcal{C}$ in $w^{*}$-topology, \phantom{\bigg)}

\item[2)] $\mathcal{W}_{\nabla\vert\CII\times\CII}^{f,\theta}$ is the unique intrinsic distance of $\lc{}L^{f,\theta},\Adm_{\mathcal{C}}\rc$ on $\CII$, \phantom{\bigg)}

\item[3)] $\big(\mathcal{C},\mathcal{W}_{\nabla}^{f,\theta}\big)$ is a complete geodesic length-metric space. \phantom{\bigg)}
\end{itemize}
\end{cor}
\begin{proof}
Let $\mathcal{C}\subset (\SII(A),\mathcal{W}_{\nabla}^{f,\theta})$. Equation \ref{EQ.COR.QOT_Distance_AC_I_2} shows $1)$ and $2)$ alike. We see $(\mathcal{C},\mathcal{W}_{\nabla}^{f,\theta})$ is a length-metric space. Furthermore, maximality of $\CII$ and $4)$ in Theorem \ref{THM.QOT_Distance} imply it is geodesic. Thus $3)$ follows if we show its completeness.\par
Let $\lset\mu^{n}\rset_{n\in\mathbb{N}}\subset\mathcal{C}$ be a Cauchy sequence. Using $1)$ in Theorem \ref{THM.QOT_Distance}, get $\mu\in\overline{\SII(A)}$ s.t.~$\mu=w^{*}$-$\lim_{n\in\mathbb{N}}\mu^{n}$. Since $\mu^{n}\sim\mu^{m}$ for all $n,m\in\mathbb{N}$, Equation \ref{EQ.PRP.Continuity_Equation_Mass_Preservation_2} further implies 

\begin{align}\label{EQ.COR.QOT_Distance_AC_II_1}
\mu^{n}(1_{A_{j}})=\mu^{m}(1_{A_{j}})
\end{align}

\noindent for all $j,n,m\in\mathbb{N}$. Using $1.1)$ in Proposition \ref{PRP.AF_Cstar_Trace_Dualisation_II}, Equation \ref{EQ.COR.QOT_Distance_AC_II_1} lets us calculate

\begin{align}\label{EQ.COR.QOT_Distance_AC_II_2}
\|\mu\|_{A^{*}}=\lim_{j\in\mathbb{N}}\hspace{0.025cm} \mu(1_{A_{j}})=\lim_{j\in\mathbb{N}}\hspace{0.025cm} \lim_{n\in\mathbb{N}}\hspace{0.025cm} \mu^{n}(1_{A_{j}})=\lim_{j\in\mathbb{N}}\hspace{0.025cm} \mu^{1}(1_{A_{j}})=\|\mu^{1}\|_{A^{*}}=1. 
\end{align}

\noindent Equation \ref{EQ.COR.QOT_Distance_AC_II_2} shows $\mu\in\SII(A)$. For all $\varepsilon>0$, there exists $n_{\varepsilon}\in\mathbb{N}$ s.t.~$\mathcal{W}_{\nabla}^{f,\theta}\lc\mu^{n},\mu^{m}\rc{}<\varepsilon$ for all $n,m\geq n_{\varepsilon}$. For all $\varepsilon>0$ and $m\geq n_{\varepsilon}$, l.s.c.~in $w^{*}$-topology as per $3)$ in Theorem \ref{THM.QOT_Distance} implies

\begin{align}\label{EQ.COR.QOT_Distance_AC_II_3}
\mathcal{W}_{\nabla}^{f,\theta}\lc\mu,\mu^{m}\rc\leq\liminf_{n\in\mathbb{N}}\hspace{0.0675cm} \mathcal{W}_{\nabla}^{f,\theta}\lc\mu^{n},\mu^{m}\rc{}<\varepsilon.
\end{align}

\noindent Equation \ref{EQ.COR.QOT_Distance_AC_II_3} shows $\lim_{m\in\mathbb{N}}\mathcal{W}_{\nabla}^{f,\theta}\lc\mu,\mu^{m}\rc{}=0$. We obtain $\mu\in\CII$.
\end{proof}

We formalise here, to the extend necessary for the study of metric geometry, energy functionals being $\Gamma$-limits w.r.t.~the coarse graining process as existence of sufficient minimising geodesics approximated in finite dimensions. Motivated by the sequential descriptions as per Definition \ref{DFN.Energy_Functional_Lower_Upper_Limits} and used in Theorem \ref{THM.Energy_Functional_Representation}, Definition \ref{DFN.QOT_Minimiser_Approximation} gives finite-dimensional approximation of minimising geodesics. We consider closure of

\begin{align}\label{EQ.SSEC.QOT_DT_MG_1}
\textrm{Geo}_{0}:=\bigcup_{j\in\mathbb{N}}\Geoj\subset\Geo
\end{align}

\noindent w.r.t.~suitable notion of convergence. Note $2.1)$ in Proposition \ref{PRP.QOT_Distance_Geodesics} shows inclusion used in Equation \ref{EQ.SSEC.QOT_DT_MG_1}. Theorem \ref{THM.QOT_Minimiser_Approximation} gives existence of sufficient minimising geodesics approximated in finite dimensions. For details on the coarse graining process, we refer to Subsection \ref{SSEC.QOT_CG}.

\begin{dfn}\label{DFN.QOT_Minimiser_Approximation}
Let $\mu^{0},\mu^{1}\in\SII(A)$ s.t.~$\mathcal{W}_{\nabla}^{f,\theta}\lc\mu^{0},\mu^{1}\rc{}<\infty$. We call $(\mu,w)\in\Geo\lc\mu^{0},\mu^{1}\rc$ approximated in finite dimensions if there exists $m\in\mathbb{N}$ and $\lc\mu^{j},w^{j}\rc_{j\geq m}\subset\Geo_{0}$ s.t.~

\begin{itemize}
\item[1)] $\lc\mu^{j},w^{j}\rc\in\Geoj\big(\bar{\mu}_{j}^{0},\bar{\mu}_{j}^{1}\big)$ for all $j\geq m$,

\item[2)] $\lc\mu^{j},w^{j}\rc_{j\geq m}$ has subsequence converging to $(\mu,w)$ in $\Admnullone$.
\end{itemize}
\end{dfn}

\begin{thm}\label{THM.QOT_Minimiser_Approximation}
Let $(\phi,\bpsi,\gamma,\nabla)$ be noncommutative differential structure for tracial AF-$C^{*}$-algebras $(A,\tau)$ and $(B,\omega)$ in $\lc{}f,\theta\rc$-setting. If $\mu^{0},\mu^{1}\in\SII(A)$ and $\mathcal{W}_{\nabla}^{f,\theta}\lc\mu^{0},\mu^{1}\rc{}<\infty$, then there exists $(\mu,w)\in\Geo\lc\mu^{0},\mu^{1}\rc$ approximated in finite dimensions.
\end{thm}
\begin{proof}
Let $\mu^{0},\mu^{1}\in\SII(A)$ s.t.~$\mathcal{W}_{\nabla}^{f,\theta}\lc\mu^{0},\mu^{1}\rc{}<\infty$. Apply $3)$ and $4)$ in Theorem \ref{THM.QOT_Distance} to get $m\in\mathbb{N}$ s.t.~for all $j\geq m$, we have

\begin{align}\label{EQ.THM.QOT_Minimiser_Approximation_1}
\Geoj\big(\bar{\mu}_{j}^{0},\bar{\mu}_{j}^{1}\big)\neq\emptyset.
\end{align}

\noindent For all $j\geq m$, let $\lc\mu^{j},w^{j}\rc\in\Geoj(\bar{\mu}_{j}^{0},\bar{\mu}_{j}^{1})$. Using $1)$ in Proposition \ref{PRP.Energy_Functional_Restriction} and further $3)$ in Theorem \ref{THM.QOT_Distance}, i.e.~Equation \ref{EQ.THM.QOT_Distance_2}, we calculate

\begin{align}\label{EQ.THM.QOT_Minimiser_Approximation_2}
\liminf_{j\in\mathbb{N}}\hspace{0.025cm} E^{f,\theta}\lc\mu^{j},w^{j}\rc{}=\liminf_{j\in\mathbb{N}}\hspace{0.0675cm} \mathcal{W}_{\nabla_{\hspace{-0.055cm} j}}^{f,\theta}\big(\bar{\mu}_{j}^{0},\bar{\mu}_{j}^{1}\big)^{2}=\mathcal{W}_{\nabla}^{f,\theta}\lc\mu^{0},\mu^{1}\rc^{2}<\infty.
\end{align}

\noindent Equation \ref{EQ.THM.QOT_Minimiser_Approximation_2} ensures we may extract suitable subsequence. Using $2)$ in Lemma \ref{LEM.Energy_Functional_II} for $t_{0}=0$, get subsequence of $\lc\mu^{j},w^{j}\rc_{j\geq m}$ converging to a $(\mu,w)\in\Admnullone$. Using $1)$ in Lemma \ref{LEM.Energy_Functional_II}, we obtain $(\mu,w)\in\Geo\lc\mu^{0},\mu^{1}\rc$ as claimed.
\end{proof}


\subsubsection*{The interpolation parameter}

We view each symmetric representing function $f$ as determining a class of energetic structures with $\theta\in [0,1]$ as interpolation parameter. Proposition \ref{PRP.QOT_Distance_Interpolation_Parameter} shows $\theta=0$ gives quantum $\lc{}-1,2\rc$-Sobolev distance independent of $f$. In the logarithmic mean setting, variation of $\theta\in [0,1]$ interpolates between, due to independence from $f$, non-geometric quantum $\lc{}-1,2\rc$-Sobolev distances and quantum $L^{2}$-Wasserstein distances. This follows the classical case \cite{ART.Dol_Naz_Sav.2009.Generalised_OT}.

\begin{prp}\label{PRP.QOT_Distance_Interpolation_Parameter}
Let $(\phi,\bpsi,\gamma,\nabla)$ be noncommutative differential structure for tracial AF-$C^{*}$-algebras $(A,\tau)$ and $(B,\omega)$ in $\lc{}f,0\rc$-setting. For all $\mu^{0},\mu^{1}\in\SII(A)$, we have

\begin{align}\label{EQ.PRP.QOT_Distance_Interpolation_Parameter_1}
\mathcal{W}_{\nabla}^{f,0}\lc\mu^{0},\mu^{1}\rc{}=\sup\hspace{0.025cm} \lset{} \babsv{1}{\lc\mu^{1}-\mu^{0}\rc{}(x)}\ \vset\ x\in A_{0},\ \|\nabla x\|_{\omega}\leq 1\rset{}.
\end{align}
\end{prp}


\pagebreak


\begin{proof}
We reduce to the finite-dimensional setting. Assume Equation \ref{EQ.PRP.QOT_Distance_Interpolation_Parameter_1} holds in the latter. For all $\mu^{0},\mu^{1}\in\SII(A)$, we use $3)$ in Theorem \ref{THM.QOT_Distance} to calculate

\begin{align*}
\mathcal{W}^{f,0}\lc\mu^{0},\mu^{1}\rc{} & = \limsup_{j\in\mathbb{N}}\hspace{0.025cm} \mu^{0}(1_{A_{j}})^{-1}\sup\hspace{0.025cm} \lset\babsv{1}{\lc\mu^{1}-\mu^{0}\rc{}(x)}\ \vset\ x\in A_{j},\ \|\nabla x\|_{\omega}\leq 1\rset \phantom{\bigg)} \\
& = \sup_{j\in\mathbb{N}}\hspace{0.025cm} \sup\hspace{0.025cm} \lset\babsv{1}{\lc\mu^{1}-\mu^{0}\rc{}(x)}\ \vset\ x\in A_{j},\ \|\nabla x\|_{\omega}\leq 1\rset \phantom{\bigg)} \\
& = \sup\hspace{0.025cm} \lset\babsv{1}{\lc\mu^{1}-\mu^{0}\rc{}(x)}\ \vset\ x\in A_{0},\ \|\nabla x\|_{\omega}\leq 1\rset{}. \phantom{\bigg)}
\end{align*}

\noindent It suffices to show Equation \ref{EQ.PRP.QOT_Distance_Interpolation_Parameter_1} in the finite-dimensional setting.\par
Assume $A$ and $B$ are finite-dimensional. Equation \ref{EQ.PRP.QOT_Distance_Interpolation_Parameter_2} below states $\mu$, $\eta$ and $f$ are irrelevant if $\theta=0$. This follows since their contributions are perturbed noncommutative division operators to the power of zero, i.e.~the identity operator. For all $\mu,\eta\in\SII(A)$ and $w\in B^{*}$, we have

\begin{align}\label{EQ.PRP.QOT_Distance_Interpolation_Parameter_2}
\mathcal{I}^{f,0}(\mu,\eta,w)=\sup_{j\in\mathbb{N}}\hspace{0.025cm} \|w\|_{\omega}^{2}\in [0,\infty].
\end{align}

\noindent For all $\mu^{0},\mu^{1}\in\SII(A)$, set 

\begin{align}\label{EQ.PRP.QOT_Distance_Interpolation_Parameter_3}
d\lc\mu^{0},\mu^{1}\rc{}:=\sup\hspace{0.025cm} \lset\babsv{1}{\lc\mu^{1}-\mu^{0}\rc{}(x)}\ \vset\ x\in A,\ \|\nabla x\|_{\omega}\leq 1\rset{}.
\end{align}

\noindent Let $\mu^{0},\mu^{1}\in\SII(A)$ s.t.~$d\lc\mu^{0},\mu^{1}\rc{}<\infty$. Finiteness implies $\mu^{1}(x)=\mu^{0}(x)$ for all $x\in\ker\nabla$ by scaling with strictly positive constants. We therefore define bounded linear functional $F_{\mu^{0},\mu^{1}}:\im\nabla\cong\ker\nabla^{\perp}\longrightarrow\mathbb{C}$ by setting 

\begin{align}\label{EQ.PRP.QOT_Distance_Interpolation_Parameter_4}
F_{\mu^{0},\mu^{1}}(x):=\lc\mu^{1}-\mu^{0}\rc{}(x)    
\end{align}

\noindent for all $\nabla x\in\im\nabla$. Equation \ref{EQ.PRP.QOT_Distance_Interpolation_Parameter_4} determines unique $w\in\im\nabla$ s.t.~$\|w\|_{\omega}=d\lc\mu^{0},\mu^{1}\rc$ and $\lc\mu^{1}-\mu^{0}\rc{}(x)=\lgl w,\nabla x\rgl_{\omega}$ for all $x\in A$. We define $(\mu,w)\in\Admnullone\lc\mu^{0},\mu^{1}\rc$ by setting

\begin{align}\label{EQ.PRP.QOT_Distance_Interpolation_Parameter_5}
\mu(t):=\lc{}1-t\rc\mu^{0}+t\mu^{1},\ w(t):=w
\end{align}

\noindent for all $t\in [0,1]$. Equation \ref{EQ.PRP.QOT_Distance_Interpolation_Parameter_2} and Equation \ref{EQ.PRP.QOT_Distance_Interpolation_Parameter_5} imply

\begin{align}\label{EQ.PRP.QOT_Distance_Interpolation_Parameter_6}
\mathcal{W}_{\nabla}^{f,0}\lc\mu^{0},\mu^{1}\rc\leq L^{f,\theta}(\mu,w)=\|w\|_{\omega}=d\lc\mu^{0},\mu^{1}\rc{}.
\end{align}


\pagebreak


We show the converse. If $(\mu,w)\in\Admnullone\lc\mu^{0},\mu^{1}\rc$, then Equation \ref{EQ.PRP.QOT_Distance_Interpolation_Parameter_2} shows

\begin{align}\label{EQ.PRP.QOT_Distance_Interpolation_Parameter_7}
\int_{0}^{1}\|w(t)\|_{\omega}dt=L^{f,0}(\mu,w).
\end{align}

\noindent Equation \ref{EQ.PRP.QOT_Distance_Interpolation_Parameter_7} in turn shows

\begin{align}\label{EQ.PRP.QOT_Distance_Interpolation_Parameter_8}
\babsv{1}{\lc\mu^{1}-\mu^{0}\rc{}(x)}\leq\int_{0}^{1}\|w(t)\|_{\omega}dt=L^{f,0}(\mu,w)
\end{align}

\noindent for all $x\in A$ s.t.~$\|\nabla x\|_{\omega}\leq 1$. Take the infimum over all admissible paths with marginals $\mu^{0}$ and $\mu^{1}$ in Equation \ref{EQ.PRP.QOT_Distance_Interpolation_Parameter_8}, followed by the supremum over all $x\in A$ s.t.~$\|\nabla x\|_{\omega}\leq 1$. This yields the converse to Equation \ref{EQ.PRP.QOT_Distance_Interpolation_Parameter_6}. Note our use of Corollary \ref{COR.Length_Functional_Reparametrisation}. Equation \ref{EQ.PRP.QOT_Distance_Interpolation_Parameter_1} holds in the finite-dimensional setting. The general case follows as discussed above.
\end{proof}


\subsection{Fundamental example classes}\label{SSEC.QOT_DT_BSP}

We provide fundamental example classes. We specify neither symmetric representing function nor interpolation parameter. First, we use generalised discrete derivatives to construct quantum optimal transport distances for tracial AF-$C^{*}$-algebras parame\-trised over finite sets. This generalises the discrete cases \cite{ART.Maa.2011.Discrete_OT_Markov}\cite{ART.Mie.2011.Discrete_OT_RctDiff} and those using internal quantum gradients. Secondly, we use dynamic quantum gradients to construct quantum optimal transport distances for tracial AF-$C^{*}$-algebras generating hyperfinite factors of type I and II by $\sigma$-weak closure. These are common algebras of observables in quantum statistical mechanics \cite{BK.Bra.1987.OpAlg_Quantum_StM_I}\cite{BK.Bra.1987.OpAlg_Quantum_StM_II}\cite{BK.Nes_Sto.2006.Rel_Ent}.\par
In the non-twisted case, we have an iterative construction. Self-adjoint unbounded operator with compact resolvent induce examples for type I-factors. We extend to the type II$_{1}$-factor using natural extensions of bounded operators on separable Hilbert space to elements in CAR-algebras \cite{BK.Nes_Sto.2006.Rel_Ent} under Clifford representations \cite{BK.Gra_Var_Fig.2001.NCG_Elements}\cite{BK.Ply_Rob.1994.Clifford_Algebras}. We tensor both to the type II$_{\infty}$-factor. In the twisted case, we show intertwining sets of Clifford generators yield direct sums of dynamic quantum gradients for tracial AF-$C^{*}$-algebras closing to the type II$_{\infty}$-factor. In the logarithmic mean setting, the non-twisted and twisted case have non-negative, resp.~strictly positive lower Ricci bounds. Thirdly, examples using non-twisted dynamic quantum gradients are given by first and second quantisation of spectral triples \cite{ART.Cha_Con_vSui.2013.NCG_Inner_Fluctuations}\cite{ART.Cha_Con_vSui.2020.NCG_Second_Quantisation}\cite{BK.vSui.2015.NCG_AF_Particle_Physics}\cite{BK.Var.2006.NCG_Elements_Short}. First quantisation of spectral triples gives examples for type I-factors induced by noncommutative Dirac operators. Second quantisation of spectral triples is extension to the type II$_{1}$-factor. Finally, we outline how second quantisation of spectral triples yields our ansatz to study noncommutative gauge theories \cite{ART.Cha_Con.1996.NCG_Spectral_Action_I}\cite{ART.Cha_Con.1997.NCG_Spectral_Action_II}\cite{ART.Cha_Con_Mar.2007.NCG_Standard_Model_Recovered}\cite{BK.vSui.2015.NCG_AF_Particle_Physics}\cite{BK.Var.2006.NCG_Elements_Short} if we generalise to quantum optimal transport parametrised by gauge fields. We view quantum optimal transport as the pointwise case. We therefore see our discussion lies in the intersection of noncommutative gauge theory, quantum statistical mechanics and quantum information theory \cite{BK.Nie_Chu.2000.Quantum_Computation_Information}.\par


\pagebreak


Standard references for factor $W^{*}$-algebras, in particular hyperfinite ones, are \cite{BK.Ped.2018.Cstar_Algebras} and \cite{BK.Tak.1979.OpAlg_I}\cite{BK.Tak.2003.OpAlg_II}\cite{BK.Tak.2003.OpAlg_III}. We refer to \cite{BK.Nes_Sto.2006.Rel_Ent} for details on CAR-algebras, as well as \cite{BK.Gra_Var_Fig.2001.NCG_Elements} and \cite{BK.Ply_Rob.1994.Clifford_Algebras} for Clifford representations over anti-symmetric Fock space. Standard references for noncommutative geometry are \cite{BK.Gra_Var_Fig.2001.NCG_Elements}\cite{BK.Var.2006.NCG_Elements_Short} and \cite{BK.vSui.2015.NCG_AF_Particle_Physics}. Whereas \cite{BK.Gra_Var_Fig.2001.NCG_Elements} provides a rather comprehensive treatment, note \cite{BK.Var.2006.NCG_Elements_Short} gives a condensed version of the former. Standard references for quantum statistical mechanics are \cite{BK.Bra.1987.OpAlg_Quantum_StM_I}\cite{BK.Bra.1987.OpAlg_Quantum_StM_II}, \cite{BK.Dav.1976.Quantum_Markov_SG}, \cite{BK.Gar_Zol.2004.Quantum_Noise}, \cite{BK.Ohy_Pet.1993.Rel_Ent} and \cite{BK.Ste_vLee.2013.Full_Quantum_StM}.


\subsubsection*{Generalised discrete derivatives over finite sets}

We use generalised discrete derivatives to construct quantum optimal transport distances in Example \ref{BSP.QOT_KL_Parametrised} for tracial AF-$C^{*}$-algebras parametrised over finite sets. This generalises the discrete cases \cite{ART.Maa.2011.Discrete_OT_Markov}\cite{ART.Mie.2011.Discrete_OT_RctDiff} and those using internal quantum gradients.\par
Let $X$ be a finite set and $u\in C(X)_{+}$. We define f.s.n.~trace $\nu_{u}$ on $C(X)$ by setting

\begin{align}\label{EQ.SSEC.QOT_DT_BSP_1}
\nu_{u}(F):=\sum_{x\in X}F(x)u(x)
\end{align}

\noindent for all $F\in C(X)$. We have finite tracial AF-$C^{*}$-algebra $\lc{}C(X),\nu_{u}\rc$ as per Example \ref{BSP.AF_Cstar_Trace_Fin}. Let $(A,\tau)$ be a tracial AF-$C^{*}$-algebra. Since $\absv{1.115}{X}<\infty$, note $C\lc{}X,A\rc\cong C(X)\otimes A\cong A^{\absv{1.115}{X}}$ as AF-$C^{*}$-algebras. Proposition \ref{PRP.Wstar_Derivation_QG_TP} yields tracial AF-$C^{*}$-algebra $\lc{}C\lc{}X,A\rc{},\nu_{u}\otimes\tau\rc$ in $C\lc{}X,L^{\infty}(A,\tau)\rc$ generated by $\lset{}C\lc{}X,A_{j}\rc\rset_{j\in\mathbb{N}}$. We have f.s.n.~trace $\nu_{u}\otimes\tau$ on $C\lc{}X,L^{\infty}(A,\tau)\rc$ given by

\begin{align}\label{EQ.SSEC.QOT_DT_BSP_2}
\big(\nu_{u}\otimes\tau\big)(F)=\sum_{x\in X}\tau\big(F(x)\big)u(x)   
\end{align}

\noindent for all $F\in C\lc{}X,L^{\infty}(A,\tau)\rc_{+}$.

\begin{bsp}\label{BSP.QOT_KL_Parametrised}
Let $X$ be a finite set and $K\in C\lc{}X\times X\rc_{+}$ an irreducible Markov kernel with steady state $u_{K}\in C(X)_{+}$ having full support. Let $(A,\tau)$ be a strongly unital tracial AF-$C^{*}$-algebra s.t.~$\tau<\infty$. We tensor $\lc{}C(X),\nu_{u_{K}}\rc$ and $(A,\tau)$ as per Equation \ref{EQ.SSEC.QOT_DT_BSP_2}. We likewise tensor $\lc{}C\lc{}X\times X\rc{},\nu_{K}\rc$ and $(A\otimes A,\tau\otimes\tau)$.\par
We have f.s.n.~trace $\tau_{K}$ on $C\lc{}X,L^{\infty}(A,\tau)\rc$ given by

\begin{align}\label{EQ.BSP.QOT_KL_Parametrised_1}
\tau_{K}(F):=\big(\nu_{u_{K}}\otimes\tau\big)(F)=\sum_{x\in X}\tau\big(F(x)\big)u_{K}(x)
\end{align}

\noindent for all $F\in C\lc{}X,L^{\infty}(A,\tau)\rc_{+}$, as well as f.s.n.~trace $\omega_{K}$ on $C\lc{}X\times X,L^{\infty}(A,\tau)\otimes L^{\infty}(A,\tau)\rc_{+}$ given by

\begin{align}\label{EQ.BSP.QOT_KL_Parametrised_2}
\omega_{K}(G):=\big(\nu_{K}\otimes(\tau\otimes\tau)\big)(G)=\sum_{x,y\in X}(\tau\otimes\tau)\big(G(x,y)\big)K(x,y)
\end{align}

\noindent for all $G\in C\lc{}X\times X,L^{\infty}(A,\tau)\otimes L^{\infty}(A,\tau)\rc_{+}$. Altogether, we obtain strongly unital tracial AF-$C^{*}$-algebra $\lc{}C\lc{}X,A\rc{},\tau_{K}\rc$ generated by $\lset{}C\lc{}X,A_{j}\rc\rset_{j\in\mathbb{N}}$, as well as $\lc{}C\lc{}X\times X,A\otimes A\rc{},\omega_{K}\rc$ generated by $\lset{}C\lc{}X\times X,A_{j}\odot A_{j}\rc\rset_{j\in\mathbb{N}}$.\par


\pagebreak


We know Equation \ref{EQ.BSP.QOT_KL_Parametrised_1} shows $L^{2}\lc{}C\lc{}X,A\rc{},\tau_{K}\rc{}=C\lc{}X,L^{2}(A,\tau)\rc$ and Equation \ref{EQ.BSP.QOT_KL_Parametrised_2} shows $L^{2}\lc{}C\lc{}X\times X,A\otimes A\rc{},\omega_{K}\rc{}=C\lc{}X\times X,L^{2}(A\otimes A,\tau\otimes\tau)\rc$. Equation \ref{EQ.SSEC.NCDS_NCG_QG_7} further shows we define local $^{*}$-homomorphisms $\phi,\bpsi:C\lc{}X,A\rc\longrightarrow C\lc{}X\times X,A\otimes A\rc$ by setting

\begin{align}\label{EQ.BSP.QOT_KL_Parametrised_3}
\phi(F)(x,y):=\phi^{\Int}\big(F(x)\big),\ \bpsi(F)(x,y):=\bpsi^{\Int}\big(F(y)\big)
\end{align}

\noindent for all $F\in C\lc{}X,A\rc$ and $x,y\in X$. Finally, pointwise algebra involution defines anti-linear isometric involution $\gamma:C\lc{}X\times X,L^{2}(A\otimes A,\tau\otimes\tau)\rc\longrightarrow C\lc{}X\times X,L^{2}(A\otimes A,\tau\otimes\tau)\rc$. We have AF-$C\lc{}X,A\rc$-bimodule structure $(\phi,\bpsi,\gamma)$ on $C\lc{}X\times X,A\otimes A\rc$.\par
Let $\lambda\geq 0$. Following Equation \ref{EQ.BSP.QOT_KL_Parametrised_3}, we define the $\lc{}K,\lambda\rc$-parametrised quantum gradient $\nabla_{K}^{\lambda}:C\lc{}X,A_{0}\rc\longrightarrow C\lc{}X\times X,L^{2}(A\otimes A,\tau\otimes\tau)\rc$ by setting

\begin{align}\label{EQ.BSP.QOT_KL_Parametrised_4}
\lc\nabla_{K}^{\lambda}F\rc(x,y):=\sqrt{\frac{\lambda}{2\tau(1_{A})}}\cdot \big(F(x)\otimes 1_{A}-1_{A}\otimes F(y)\big)
\end{align}

\noindent for all $F\in C\lc{}X,A_{0}\rc$ and $x,y\in X$. We have $C\lc{}X\times X,A\otimes A\rc\cong C\lc{}X,A\rc\otimes C\lc{}X,A\rc$. Using the latter and up to positive constant, $\nabla_{K}^{\lambda}$ is the generalised discrete derivative on $C\lc{}X,A\rc$ as per Definition \ref{DFN.CWstar_Derivation_Generalised_Discrete} restricted to $C\lc{}X,A_{0}\rc$. Proposition \ref{PRP.CWstar_Derivation_Generalised_Discrete} shows said generalised discrete derivative is a bounded symmetric $C\lc{}X,A\rc$-module derivation. Equation \ref{EQ.BSP.QOT_KL_Parametrised_4} shows $\nabla_{K}^{\lambda}$ commutes with Hilbert space projections to generating $C^{*}$-subalgebras. Thus $\nabla_{K}^{\lambda}$ is a quantum gradient. If $(A,\tau)=\lc\mathbb{C},1\rc$, then Equation \ref{EQ.BSP.QOT_KL_Parametrised_4} specialises to the discrete derivative. If $\absv{1.115}{X}=1$, then Equation \ref{EQ.BSP.QOT_KL_Parametrised_4} instead specialises to the $\lambda$-internal quantum gradient on $A$ as per Definition \ref{DFN.Wstar_Derivation_QG_Internal}. If $\absv{1.115}{X}>1$, then $\nu_{K}\neq\nu_{u_{K}}\otimes\nu_{u_{K}}$ since $K$ and $u_{K}$ are stochastic. Hence $\nabla_{K}^{\lambda}$ is internal quantum gradient if and only if $\absv{1.115}{X}=1$. Parametrised quantum gradients as per Equation \ref{EQ.BSP.QOT_KL_Parametrised_4} therefore generalise discrete derivatives and internal quantum gradients by using $(A,\tau)=\lc\mathbb{C},1\rc$, resp.~$\absv{1.115}{X}=1$.\par
We obtain noncommutative differential structures which define quantum optimal transport distances of discrete densities evaluating in tracial AF-$C^{*}$-algebras. If we use $(A,\tau)=\lc\mathbb{C},1\rc$ here with $K$ as in \cite{ART.Maa.2011.Discrete_OT_Markov}, then we recover discrete Wasserstein distances associated to Markov chains with detailed balance condition \cite{ART.Maa.2011.Discrete_OT_Markov}. We likewise recover \cite{ART.Mie.2011.Discrete_OT_RctDiff}. In summary, we generalise the discrete cases \cite{ART.Maa.2011.Discrete_OT_Markov}\cite{ART.Mie.2011.Discrete_OT_RctDiff} and any using internal quantum gradients. We recover these by using trivial codomain, resp.~domain.
\end{bsp}


\subsubsection*{Dynamic quantum gradients for hyperfinite factors of type I and II}

We use dynamic quantum gradients to construct quantum optimal transport distances for tracial AF-$C^{*}$-algebras generating hyperfinite factors of type I and II by $\sigma$-weak closure \cite{BK.Bra.1987.OpAlg_Quantum_StM_I}\cite{BK.Bra.1987.OpAlg_Quantum_StM_II}\cite{BK.Nes_Sto.2006.Rel_Ent}. The iterative construction of non-twisted dynamic quantum gradients is given by following Example \ref{BSP.QOT_Type_I}, Example \ref{BSP.QOT_Type_II_1} and Example \ref{BSP.QOT_Type_II_Infty} in order. Note Example \ref{BSP.QOT_Second_Quantisation_Parametrised} clarifies their importance. We construct direct sums of twisted dynamic quantum gradients, each induced by a Clifford generator, in Example \ref{BSP.QOT_Type_II_Twisted}. In the logarithmic mean setting, Example \ref{BSP.L2W_Ric_Wstar_Derivation_QG_Dynamic_System} and Example \ref{BSP.L2W_Ric_Wstar_Derivation_QG_Intertwining_Clifford} in Subsection \ref{SSEC.L2W_EVI_Ric}, both of which use Theorem \ref{THM.L2W_Ric}, imply all examples constructed here have non-negative lower Ricci bounds. Example \ref{BSP.L2W_Ric_Wstar_Derivation_QG_Intertwining_Clifford} shows strict positivity in the twisted case.\par
We give the iterative construction of non-twisted dynamic quantum gradients. Each step in the construction is induced by trace-preserving local $C^{*}$-dynamical systems as per Corollary \ref{COR.Wstar_Derivation_QG_Dynamic_System}. We cover type I-factors in Example \ref{BSP.QOT_Type_I}, the type II$_{1}$-factor in Example \ref{BSP.QOT_Type_II_1}, and the type II$_{\infty}$-factor Example \ref{BSP.QOT_Type_II_Infty}. We apply Example \ref{BSP.QOT_Type_I} to get first, and Example \ref{BSP.QOT_Type_II_1} to get second quantisation of spectral triples.

\begin{bsp}\label{BSP.QOT_Type_I}
Hyperfinite factors of type I are of form $\BII(H)$ for a separable Hilbert space $H$. Let $H$ be a separable Hilbert space, $D\in\UBII(H)_{h}$ with compact resolvent and $\{e_{j}\}_{j\in\mathbb{N}}$ an orthonormal eigenbasis of the latter. For all $j\in\mathbb{N}$, let $P_{j}:H\longrightarrow\langle e_{1},\ldots,e_{j}\rangle_{\mathbb{C}}$ be the Hilbert space projection. The orthonormal eigenbasis determines unique unitary $U:H\longrightarrow\ell^{2}\lc\mathbb{N}\rc$. For all $j\in\mathbb{N}$, set $H_{j}:=P_{j}H$, $U_{j}:=\com_{H_{j}}U=P_{j}UP_{j}$ and

\begin{align}\label{EQ.BSP.QOT_Type_I_1}
A_{j}:=\BII(H_{j})=U_{j}^{*}M_{j}(\mathbb{C})U_{j}.    
\end{align}

\noindent We have tracial AF-$C^{*}$-algebra $\lc\KII(H),\tr\rc$ in $\BII(H)$ generated by $\lset{}A_{j}\rset_{j\in\mathbb{N}}$. We equip $\KII(H)$ with its canonical AF-$\KII(H)$-bimodule structure.\par
For all $j\in\mathbb{N}$, Equation \ref{EQ.BSP.QOT_Type_I_1} shows

\begin{align}\label{EQ.BSP.QOT_Type_I_2}
\textrm{Ad}_{t}^{D}(A_{j})\subset A_{j}
\end{align}

\noindent for all $t\in\mathbb{R}$. Equation \ref{EQ.BSP.QOT_Type_I_2} shows we have $\tr$-preserving local $C^{*}$-dynamical system $\lc\KII(H),\mathbb{R},\restr{0.925}{\Ad^{D}}{\KII(H)}\rc$. Note $L^{2}\lc\KII(H),\tr\rc{}=S^{2}(H)$ for the Hilbert-Schmidt operators on $H$ \cite{BK.Bla.2006.OpAlg}. We apply Lemma \ref{LEM.Wstar_Derivation_QG_Dynamic_System} to get unique $\mathcal{D}_{\Ad}\in\UBII\lc{}S^{2}(H)\rc$ s.t.~

\begin{align}\label{EQ.BSP.QOT_Type_I_3}
L_{\textrm{Ad}_{t}^{D}(x)}=\textrm{Ad}_{t}^{\mathcal{D}_{\Ad}}(L_{x})
\end{align}

\noindent for all $t\in\mathbb{R}$ and $x\in\KII(H)$. By norm differentiation of Equation \ref{EQ.BSP.QOT_Type_I_3}, Corollary \ref{COR.Wstar_Derivation_QG_Dynamic_System} yields non-twisted dynamic quantum gradient given by

\begin{align}\label{EQ.BSP.QOT_Type_I_4}
\nabla^{\mathcal{D}_{\Ad}}x=i\lb\mathcal{D}_{\Ad},x\rb_{A}=iL^{-1}\lc\overline{\mathcal{D}_{\Ad}L_{x}-L_{x}\mathcal{D}_{\Ad}}\rc{}  
\end{align}

\noindent for all $x\in\KII(H)_{0}$. Equation \ref{EQ.BSP.QOT_Type_I_3} and Equation \ref{EQ.BSP.QOT_Type_I_4} show $\mathcal{D}_{\Ad}$ is $\id_{\KII(H)}$-intertwining. We know the identities for dynamic quantum gradient, its adjoint and finally Laplacian given in Corollary \ref{COR.Wstar_Derivation_QG_Intertwining_I} and Corollary \ref{COR.Wstar_Derivation_QG_Intertwining_II} apply.\par
We pull back along $L$ to $\KII(H)$ as follows. For all $j\in\mathbb{N}$, note Equation \ref{EQ.BSP.QOT_Type_I_1} shows $D$ is $H_{j}$-reducible and set $D_{\hspace{-0.055cm} j}:=\com_{H_{j}}D=P_{j}DP_{j}$. For all $t\in\mathbb{R}$ and $j\in\mathbb{N}$, arguing as for Equation \ref{EQ.PRP.Local_Operator_1} in Proposition \ref{PRP.Local_Operator} shows

\begin{align}\label{EQ.BSP.QOT_Type_I_5}
e^{itD}=e^{itD_{\hspace{-0.055cm} j}}\oplus e^{itD_{\hspace{-0.055cm} j}^{\perp}}    
\end{align}

\noindent w.r.t.~$\BII(H_{j})\oplus\BII(H_{j}^{\perp})$.\par
Note $\KII(H)_{0}\subset\BII(H)$ is strongly dense and $\dblv{}e^{itD}\dblv_{\BII(H)}=1$ in each case. Using the latter, Equation \ref{EQ.BSP.QOT_Type_I_5} and sequential strong continuity of multiplication show

\begin{align}\label{EQ.BSP.QOT_Type_I_6}
t\mapsto L_{e^{itD}}\in\UII\lc\BII\lc{}S^{2}(H)\rc\rc{}    
\end{align}

\noindent is a strongly continuous unitary group. Equation \ref{EQ.BSP.QOT_Type_I_3} additionally shows

\begin{align}\label{EQ.BSP.QOT_Type_I_7}
L_{e^{itD}}L_{x}L_{e^{-itD}}=\textrm{Ad}_{t}^{\mathcal{D}_{\Ad}}(L_{x})
\end{align}

\noindent for all $t\in\mathbb{R}$, $j\in\mathbb{N}$ and $x\in A_{j}$. Equation \ref{EQ.BSP.QOT_Type_I_7} extends to $\KII(H)$ by norm density. Then uniqueness, Corollary \ref{COR.Wstar_Derivation_QG_Intertwining_II}, and Equation \ref{EQ.BSP.QOT_Type_I_7} imply

\begin{align}\label{EQ.BSP.QOT_Type_I_8}
\nabla^{D}x:=\nabla^{\mathcal{D}_{\Ad}}x=i\lb{}D_{\hspace{-0.055cm} j},x\rb{}=i\lc{}D_{\hspace{-0.055cm} j}x-xD_{\hspace{-0.055cm} j}\rc{}=\overline{i\lc{}Dx-xD\rc{}}
\end{align}

\noindent for all $j\in\mathbb{N}$ and $x\in A_{j}$. The identities for dynamic quantum gradient, its adjoint and Laplacian given in Corollary \ref{COR.Wstar_Derivation_QG_Intertwining_I} and Corollary \ref{COR.Wstar_Derivation_QG_Intertwining_II} pull back accordingly.\par
We obtain noncommutative differential structures which define quantum optimal transport distances of density operators. Note all constructions of non-twisted dynamic quantum gradients reduce to Equation \ref{EQ.BSP.QOT_Type_I_8} in this example. In the logarithmic mean setting, which uses $\theta=1$, accessibility components are complete geodesic length-metric spaces even for $\dim_{\mathbb{C}}H=\infty$. We use this for first quantisation of spectral triples.
\end{bsp}

Equation \ref{EQ.SSEC.QOT_DT_BSP_3} are abstract canonical anti-commutation relations of CAR-algebras \cite{BK.Nes_Sto.2006.Rel_Ent}. Clifford representations determined by Equation \ref{EQ.SSEC.QOT_DT_BSP_9} provide natural concrete realisations \cite{BK.Gra_Var_Fig.2001.NCG_Elements}\cite{BK.Ply_Rob.1994.Clifford_Algebras}. Let $H$ be a separable Hilbert space. The CAR-algebra $\AII(H)$ over $H$ is defined as the unique unital $C^{*}$-algebra, up to isometric $^{*}$-isomorphism, equipped with a bounded anti-linear map $a:H\longrightarrow\mathcal{A}(H)$ s.t.~$C^{*}\lc\im a,1_{\mathcal{A}(H)}\rc{}=\mathcal{A}(H)$ and

\begin{align}\label{EQ.SSEC.QOT_DT_BSP_3}
a(u)^{*}a(v)+a(v)a(u)^{*}=\lgl u,v\rgl_{H}\cdot 1_{\AII(H)},\ a(u)a(v)+a(v)a(u)=0
\end{align}

\noindent for all $u,v\in H$. We consider CAR-algebras as Clifford algebras here. For all $u\in H$, set $b(u):=a(u)+a(u)^{*}$. Then Equation \ref{EQ.SSEC.QOT_DT_BSP_3} are equivalent to the Clifford relations

\begin{align}\label{EQ.SSEC.QOT_DT_BSP_4}
b(u)b(v)+b(v)b(u)=2\RE\lgl u,v\rgl_{H}\cdot 1_{\AII(H)}
\end{align}

\noindent for all $u,v\in H$. Thus the universal property of Clifford algebras applies and lets us extend bounded linear maps preserving Equation \ref{EQ.SSEC.QOT_DT_BSP_3}, hence Equation \ref{EQ.SSEC.QOT_DT_BSP_4}, from $H$ to $\AII(H)$ \lc{}cf.~Proposition 5.1 in \cite{BK.Gra_Var_Fig.2001.NCG_Elements}\rc{}. For all $\varphi\in\UII\lc\BII(H)\rc$, said universal property determines the Bogoliubov automorphism $\Cliff(\varphi)\in\textrm{Aut}\lc\AII(H)\rc$ of $\varphi$ by setting

\begin{align}\label{EQ.SSEC.QOT_DT_BSP_5}
\Cliff(\varphi)\big(a(u)\big):=a\lc\varphi(u)\rc{}    
\end{align}

\noindent for all $u\in H$ \lc{}cf.~Example 5.2 in \cite{BK.Gra_Var_Fig.2001.NCG_Elements}\rc{}.\par


\pagebreak


We determine l.s.c.~faithful semi-finite trace $\tau$ on $\AII(H)$ by setting

\begin{align}\label{EQ.SSEC.QOT_DT_BSP_6}
\tau\lc{}a\lc{}u_{1}\rc^{*}\ldots a\lc{}u_{n}\rc^{*}a\lc{}v_{m}\rc\ldots a\lc{}v_{1}\rc\rc{}:=\delta_{nm}\det\lc\frac{1}{2}\lc\lgl u_{k},v_{l}\rgl_{H}\rc_{k,l=1}^{n}\rc{}
\end{align}

\noindent for all $n,m\in\mathbb{N}$ and $\lset{}u_{k}\rset_{k=1}^{n},\lset{}v_{l}\rset_{l=1}^{m}\subset H$ \cite{BK.Nes_Sto.2006.Rel_Ent}. Note $\tau$ is the unique normalised trace on $\AII(H)$. If $\dim_{\mathbb{C}}H=n$, then $\lc\AII(H),\tau\rc\cong \lc\otimes_{k=1}^{n}M_{2}(\mathbb{C}),2^{-n}\otimes_{k=1}^{n}\tr_{2}\rc\cong \lc{}M_{2^{n}}(\mathbb{C}),2^{-n}\tr_{2^{n}}\rc$ as tracial $C^{*}$-algebras \cite{BK.Nes_Sto.2006.Rel_Ent}. If $\dim_{\mathbb{C}}H=\infty$, then $\AII(H)''$ is the hyperfinite factor of type II$_{1}$ up to choice of faithful unital $^{*}$-representation \cite{BK.Ped.2018.Cstar_Algebras}.\par
We associate faithful unital $^{*}$-representations of CAR-algebras over anti-symmetric Fock space to orthogonal complex structures. Such representations are called Clifford representations \cite{BK.Gra_Var_Fig.2001.NCG_Elements}\cite{BK.Ply_Rob.1994.Clifford_Algebras}. We equip $H$ with Euclidean structure of $\|.\|_{H}$. Let $J$ be an orthogonal complex structure on $H$. Using $J$ as imaginary unital left-multiplication, we complexify to $H[J]=H\oplus iH$. We define inner product of $H[J]$ by setting

\begin{align}\label{EQ.SSEC.QOT_DT_BSP_7}
\lgl u,v\rgl_{H[J]}:=\RE\lgl u,v\rgl_{H}+i\RE\lgl u,J(v)\rgl_{H}    
\end{align}

\noindent for all $u,v\in H[J]$. Thus $\lc{}H[J],\|.\|_{J}\rc$ is a Hilbert space. Equation \ref{EQ.SSEC.QOT_DT_BSP_7} induces inner product $\bigwedge\|.\|_{J}$ of anti-symmetric Fock space $\mathcal{F}(H[J]):=\bigwedge H[J]$ by universal property of the exterior algebra \cite{BK.Gra_Var_Fig.2001.NCG_Elements}. Hence $\lc\mathcal{F}(H[J]),\bigwedge\|.\|_{J}\rc$ is a Hilbert space. We define bounded anti-linear map $a_{J}:H\longrightarrow\BII\lc\FII(H[J])\rc$ by setting

\begin{align}\label{EQ.SSEC.QOT_DT_BSP_8}
\big(a_{J}(u)\big)^{*}(v):=u\wedge v
\end{align}

\noindent for all $u\in H$ and $v\in\FII(H[J])$. Using $\bigwedge\|.\|_{J}$ and Equation \ref{EQ.SSEC.QOT_DT_BSP_8} to obtain adjoints in $\BII\lc\FII(H[J])\rc$, we directly verify Equation \ref{EQ.SSEC.QOT_DT_BSP_3} for $a_{J}$ \lc{}cf.~pp.186-187 in \cite{BK.Gra_Var_Fig.2001.NCG_Elements}\rc{}.\par
Finally, we determine the Clifford representation $\rho_{J}:\AII(H)\longrightarrow\BII\lc\FII(H[J])\rc$ for $J$ by setting

\begin{align}\label{EQ.SSEC.QOT_DT_BSP_9}
\rho_{J}(u):=a_{J}(u)+a_{J}(u)^{*}
\end{align}

\noindent for all $u\in H$. Note $\rho_{J}(u)=\rho_{J}\lc{}b(u)\rc$ in each case since we consider $H\cong b(H)\subset\AII(H)$ as set of generators. Thus $2a_{J}(u)=\rho_{J}(u)-i\rho_{J}\lc{}J(u)\rc$ for all $u\in H$ by \lc{}anti-\rc{}linearity, hence $\rho_{J}$ is a faithful unital $^{*}$-representation s.t.~

\begin{align}\label{EQ.SSEC.QOT_DT_BSP_10}
\AII(H[J]):=\rho_{J}\lc\AII(H)\rc\cong\AII(H)
\end{align}

\noindent is CAR-algebra over $H$ and Clifford algebra of $\|.\|_{H}^{2}$. Note the unique normalised and l.s.c.~faithful semi-finite trace $\tau$ on $\AII(H[J])$ is determined by Equation \ref{EQ.SSEC.QOT_DT_BSP_6} for $a_{J}$. We have tracial $C^{*}$-algebra $\lc\AII(H[J]),\tau\rc$ in $\AII(H[J])''\subset\BII\lc\FII(H[J])\rc$.

\begin{bsp}\label{BSP.QOT_Type_II_1}
The hyperfinite factor of type II$_{1}$ is $\AII(H)''$ for a separable Hilbert space $H$. Let $H$ be a separable Hilbert space, $J$ orthogonal complex structure on $H$ and $H_{1}\subset H_{2}\subset\ldots\subset H$ Hilbert subspaces with

\begin{align}\label{EQ.BSP.QOT_Type_II_1}
H=\overline{\bigcup_{j\in\mathbb{N}}H_{j}}^{\|.\|_{H}}
\end{align}

\noindent and s.t.~

\begin{align}\label{EQ.BSP.QOT_Type_II_2}
J(H_{j})\subset H_{j}
\end{align}

\noindent for all $j\in\mathbb{N}$. Equation \ref{EQ.BSP.QOT_Type_II_2} shows $J$ is orthogonal complex structure on $H_{j}$ and

\begin{align}\label{EQ.BSP.QOT_Type_II_3}
H_{j}[J]\subset H_{j+1}[J]\subset H[J]
\end{align}

\noindent in each case. Equation \ref{EQ.BSP.QOT_Type_II_1} and Equation \ref{EQ.BSP.QOT_Type_II_3} show $H[J]=\overline{\bigcup_{j\in\mathbb{N}}H_{j}[J]}^{\|.\|_{H[J]}}$. They moreover show analogous restriction properties for $\rho_{J}$ and $\AII(H[J])$. We have tracial AF-$C^{*}$-algebra $\lc\AII(H[J]),\tau\rc$ in $\AII(H[J])''\subset\BII\lc\FII(H[J])\rc$ generated by $\lset\hspace{-0.025cm} \AII(H_{j}[J])\rset_{j\in\mathbb{N}}$. We equip $\AII(H[J])$ with its canonical AF-$\AII(H[J])$-bimodule structure.\par
We construct $\tau$-preserving local $C^{*}$-dynamical system. Let $\varphi\in\UII\lc\BII(H)\rc$ s.t.~we have $\varphi(H_{j}),\varphi^{-1}(H_{j})\subset H_{j}$ for all $j\in\mathbb{N}$. Using $\varphi^{-1}=\varphi^{*}$, we directly verify $\com_{H_{j}}\varphi\in\UII\lc\BII(H_{j})\rc$ for all $j\in\mathbb{N}$. We obtain the $J$-twisted Bogoliubov automorphism

\begin{align}\label{EQ.BSP.QOT_Type_II_4}
\textrm{Cliff}_{J}(\varphi):=\rho_{J}\circ\Cliff(\varphi)\circ\rho_{J}^{-1}\in\textrm{Aut}\lc\AII(H[J])\rc{}
\end{align}

\noindent s.t.~$\restr{0.925}{\textrm{Cliff}_{J}(\varphi)}{\AII(H_{j}[J])}=\rho_{J}\circ\Cliff\lc\com_{H_{j}}\varphi\rc\circ\rho_{J}^{-1}\in\textrm{Aut}\lc\AII(H_{j}[J])\rc$ for all $j\in\mathbb{N}$. We select compatible Dirac operator. Let $D\in\UBII(H)_{h}$ with compact resolvent and orthonormal eigenbasis $\{e_{j}\}_{j\in\mathbb{N}}$. For all $j\in\mathbb{N}$, let $H_{j}=\langle e_{1},\ldots,e_{j}\rangle_{\mathbb{C}}$. For all $t\in\mathbb{R}$, Equation \ref{EQ.BSP.QOT_Type_I_5} shows $e^{itD}\in\UII\lc\BII(H)\rc$ satisfies our assumptions on $\varphi$ in Equation \ref{EQ.BSP.QOT_Type_II_4}. For all $t\in\mathbb{R}$, set 

\begin{align}\label{EQ.BSP.QOT_Type_II_5}
\alpha_{t}:=\textrm{Cliff}_{J}\lc{}e^{itD}\rc\in\textrm{Aut}\lc\AII(H[J])\rc{}.   
\end{align}

\noindent For all $t\in\mathbb{R}$ and $j\in\mathbb{N}$, Equation \ref{EQ.SSEC.QOT_DT_BSP_6} shows $\alpha_{t}$ is $\tau$-preserving and Equation \ref{EQ.BSP.QOT_Type_II_4} shows $\alpha_{t}\lc\AII(H_{j}[J])\rc\subset\AII(H_{j}[J])$. We show strong continuity of $\alpha:\mathbb{R}\longrightarrow\textrm{Aut}\lc\AII(H[J])\rc$ to conclude. Since $\dblv{}e^{itD}\dblv_{\BII(H)}=1$ for all $t\in\mathbb{R}$ by functional calculus, we see locality of $\alpha$ as above in fact reduces to $\|.\|_{\AII(H[J])}$-continuity upon evaluation on $\AII(H[J])_{0}$. Taken together with the $^{*}$-homomorphism property, we further reduce to $\|.\|_{\AII(H[J])}$-continuity upon evaluation on $a\lc\bigcup_{j\in\mathbb{N}}H_{j}\rc$.\par


\pagebreak


For all $u\in\bigcup_{j\in\mathbb{N}}H_{j}$, Equation \ref{EQ.SSEC.QOT_DT_BSP_5} shows the map

\begin{align}\label{EQ.BSP.QOT_Type_II_6}
t\mapsto\alpha_{t}\big(a_{J}(u)\big)=\rho_{J}\lc\Cliff\lc{}e^{itD}\rc\big(a(u)\big)\rc{}=a_{J}\lc{}e^{itD}u\rc{}
\end{align}

\noindent is $\|.\|_{\AII(H[J])}$-continuous. Equation \ref{EQ.BSP.QOT_Type_II_6} implies strong continuity. We have $\tau$-preserving local $C^{*}$-dynamical system $\lc\AII(H[J]),\mathbb{R},\alpha\rc$. Corollary \ref{COR.Wstar_Derivation_QG_Dynamic_System} yields non-twisted dynamic quantum gradient. In Example \ref{BSP.QOT_Second_Quantisation}, note Equation \ref{EQ.BSP.QOT_Second_Quantisation_2} gives an explicit formula for Equation \ref{EQ.BSP.QOT_Type_II_5} using wedged conjugation group of $\absv{1.15}{D}$ and for suitable $J$ depending on eigenvalues of $D$. The formula is taken from Proposition 2.6 in \cite{ART.Cha_Con_vSui.2020.NCG_Second_Quantisation}. However, we explicitly solve the associated implementation problem \cite{BK.Ply_Rob.1994.Clifford_Algebras} in Lemma \ref{LEM.QOT_Second_Quantisation_Implementation_Fock}.\par
We obtain noncommutative differential structures which define quantum optimal transport distances of states on CAR-algebras. Note the non-twisted dynamic quantum gradients used are induced by trace-preserving local $C^{*}$-dynamical systems lifted from Example \ref{BSP.QOT_Type_I} via Clifford representations to Equation \ref{EQ.BSP.QOT_Type_II_5}. We use this to get second quantisation of spectral triples as extension of their first quantisation.
\end{bsp}

\begin{rem}\label{REM.QOT_Type_II_1}
In contrast to Example \ref{BSP.QOT_Type_I}, the construction of dynamic quantum gradients in Example \ref{BSP.QOT_Type_II_1} does not pull back along canonical left-actions. We derive explicit formula for Equation \ref{EQ.BSP.QOT_Type_II_5} in Example \ref{BSP.QOT_Second_Quantisation} using wedged conjugation groups analogous to the construction in Example \ref{BSP.QOT_Type_I}. Choice of orthogonal complex structure is necessary to solve the associated implementation problem.
\end{rem}

\begin{bsp}\label{BSP.QOT_Type_II_Infty}
The hyperfinite factor of type II$_{\infty}$ is $W^{*}$-tensor product $\BII(H)\otimes\AII(\HII)''$ for infinite-dimensional separable Hilbert spaces $H$ and $\HII$. We do not identify $H\cong\HII$ since their finite-dimensional approximation differs in general. Let $H$ be a separable Hilbert space and assume the setting of Example \ref{BSP.QOT_Type_I} for $D\in\UBII(H)_{h}$. Let $\HII$ be a separable Hilbert space and assume the setting of Example \ref{BSP.QOT_Type_II_1} for $\DII\in\UBII(\HII)_{h}$. By $1)$ in Proposition \ref{PRP.Wstar_Derivation_QG_TP}, the tensor product construction yields tracial AF-$C^{*}$-algebra $\lc\KII(H)\otimes\AII\lc\HII[J]\rc{},\tr\otimes\tau\rc$ in $\BII(H)\otimes\AII\lc\HII[J]\rc{}''$ generated by $\lset\KII(H_{j})\odot\AII\lc\HII_{j}[J]\rc\rset_{j\in\mathbb{N}}$. We equip $\KII(H)\otimes\AII\lc\HII[J]\rc$ with its canonical AF-$\KII(H)\otimes\AII\lc\HII[J]\rc$-bimodule structure.\par
We require $\tau$-preserving local $C^{*}$-dynamical system. For all $t\in\mathbb{R}$, set

\begin{align}\label{EQ.BSP.QOT_Type_II_Infty_1}
\alpha_{t}:=\textrm{Ad}_{t}^{D}\otimes\textrm{Cliff}_{J}\lc{}e^{it\DII}\rc\in\textrm{Aut}\lc\KII(H)\otimes\AII\lc\HII[J]\rc\rc{}.
\end{align}

\noindent Example \ref{BSP.QOT_Type_I} and Example \ref{BSP.QOT_Type_II_1} show we have $\tr\otimes\tau$-preserving local $C^{*}$-dynamical system $\lc\KII(H)\otimes\AII\lc\HII[J]\rc{},\mathbb{R},\alpha\rc$ by reducing to elementary tensors. Corollary \ref{COR.Wstar_Derivation_QG_Dynamic_System} yie\-lds non-twisted dynamic quantum gradient. Since the latter are furthermore defined by norm differentiation, it is the tensor product quantum gradient of the dynamic ones as per Example \ref{BSP.QOT_Type_I} and Example \ref{BSP.QOT_Type_II_1} given by Proposition \ref{PRP.Wstar_Derivation_QG_TP}.\par
We obtain noncommutative differential structures which define quantum optimal transport distances of density operators evaluating in CAR-algebras. The non-twisted dynamic quantum gradients used are induced by trace-preserving local $C^{*}$-dynamical systems which are tensor products of those in Example \ref{BSP.QOT_Type_I} and Example \ref{BSP.QOT_Type_II_1}. This finalises our three-step iterative construction.
\end{bsp}

We construct direct sums of twisted dynamic quantum gradients induced by Clifford generators. For this, we use tracial AF-$C^{*}$-algebras in the setting of Example \ref{BSP.QOT_Type_II_Infty}.

\begin{bsp}\label{BSP.QOT_Type_II_Twisted}
Let $H$ and $\HII$ be separable Hilbert spaces. Let $T\in\BII(H)_{h}$ s.t.~we have $\spec T\subset\lset\hspace{-0.025cm} \pm 1\rset$ and with orthonormal eigenbasis $\lset{}u_{j}\rset_{j\in\mathbb{N}}$. Let $\lset{}v_{j}\rset_{j\in\mathbb{N}}$ be orthonormal basis of $\HII$. Let $m\in\mathbb{N}$. For all $j\in\mathbb{N}$, set 

\begin{align}\label{EQ.BSP.QOT_Type_II_Twisted_1}
H_{j}:=\lgl u_{1},\ldots,u_{j}\rgl_{\mathbb{C}},\ \mathcal{H}_{j}:=\lgl v_{1},\ldots,v_{m-1},\ldots v_{m-1+j}\rgl_{\mathbb{C}}.
\end{align}

We use trivial orthogonal complex structure $J:=iI_{\HII}$ on $\HII$ and suppress it. Using finite-dimensional approximation given by Equation \ref{EQ.BSP.QOT_Type_II_Twisted_1} and following construction in Example \ref{BSP.QOT_Type_II_Infty}, we have tracial AF-$C^{*}$-algebra $\lc\KII(H)\otimes\AII(\HII),\tr\otimes\tau\rc$ generated by $\lset\KII(H_{j})\odot\AII\lc\HII_{j}\rc\rset_{j\in\mathbb{N}}$. We determine the principle automorphism $\varphi\in\textrm{Aut}\lc\AII(H)\rc$ of $\AII(H)$ by setting 

\begin{align}\label{EQ.BSP.QOT_Type_II_Twisted_2}
\varphi\lc\rho(u)\rc{}=-\rho(u)
\end{align}

\noindent for all $u\in\HII$. Since $\varphi$ is a self-adjoint involutive local $^{*}$-homomorphism, we know $\phi:=\id_{\BII(H)}\otimes\varphi\in\textrm{Aut}\lc\KII(H)\otimes\AII(\HII)\rc$ is one. We have AF-$\KII(H)\otimes\AII(\HII)$-bimodule structure $\lc\phi,\id_{\KII(H)\otimes\AII(\HII)},\gamma^{\phi}\rc$ on $\KII(H)\otimes\AII(\HII)$ as per $1)$ in Proposition \ref{PRP.Wstar_Derivation_QG_Dynamic}.\par
Let $C>0$. For all $n\in\lset{}1,\ldots,m\rset$, set $d_{n}:=T\otimes C^{\frac{1}{2}}\rho\lc{}v_{n}\rc$. Get $T^{2}=I_{H}$ by $\spec T\subset\lset\hspace{-0.025cm} \pm 1\rset$. Equation \ref{EQ.SSEC.QOT_DT_BSP_4} and Equation \ref{EQ.BSP.QOT_Type_II_Twisted_1} show $\lset{}d_{n}\rset_{n=1}^{m}\subset\BII(H)\otimes\AII(\HII)''$ is a $\phi$-intertwining set of Clifford generators for $C>0$ as per $1)$ in Definition \ref{DFN.Wstar_Derivation_QG_Intertwining_Clifford}. For all $n\in\lset{}1,\ldots,m\rset$, we know Corollary \ref{COR.Wstar_Derivation_QG_Intertwining_I} yields twisted dynamic quantum gradient $\partial_{n}=\nabla^{-iL_{d_{n}},\phi}$ and its Laplacian $\Delta_{n}=\mathrlap{\phantom{\partial}^{*}}\partial_{n}\partial_{n}$ as per $2)$ in Definition \ref{DFN.Wstar_Derivation_QG_Intertwining_Clifford}. Proposition \ref{PRP.Wstar_Derivation_QG_DS_I} yields direct sum quantum gradient $\nabla^{\oplus}=\oplus_{n=1}^{m}\partial_{n}:\KII(H)_{0}\odot\AII(\HII)_{0}\longrightarrow L^{2}\lc\oplus_{n=1}^{m}\KII(H)\otimes\AII(\HII),\oplus_{n=1}^{m}\tr\otimes\tau\rc$ given by

\begin{align}\label{EQ.BSP.QOT_Type_II_Twisted_3}
\nabla^{\oplus}x=\lc\partial_{1}x,\ldots,\partial_{m}x\rc{}=\lc\nabla^{-iL_{d_{1}},\phi}x,\ldots,\nabla^{-iL_{d_{m}},\phi}x\rc{}
\end{align}

\noindent for all $x\in\KII(H)_{0}\odot\AII(\HII)_{0}$. Since $\Delta^{\oplus}=\sum_{n=1}^{m}\Delta_{n}$ by $4)$ in Proposition \ref{PRP.Wstar_Derivation_QG_DS_I}, Lemma \ref{LEM.Wstar_Derivation_QG_Intertwining_Clifford} implies

\begin{align}\label{EQ.BSP.QOT_Type_II_Twisted_4}
\partial_{n}\Delta^{\oplus}=\big(\Delta^{\oplus}+4C\cdot I\big)\partial_{n}
\end{align}

\noindent for all $n\in\lset{}1,\ldots,m\rset$. Note Equation \ref{EQ.BSP.QOT_Type_II_Twisted_4} lets us apply Theorem \ref{THM.L2W_Ric} to show strictly positive lower Ricci bounds in Example \ref{BSP.L2W_Ric_Wstar_Derivation_QG_Intertwining_Clifford}. If we rescale each partial gradient of $\nabla^{\oplus}$ as $\partial_{n}\mapsto\lambda\partial_{n}$ for $\lambda>0$, then $4C$ in Equation \ref{EQ.BSP.QOT_Type_II_Twisted_4} is $\lambda\cdot 4C$ instead.\par
We obtain noncommutative differential structures which define quantum optimal transport distances of density operators evaluating in CAR-algebras. Yet in contrast to Example \ref{BSP.QOT_Type_II_Infty}, we use direct sum quantum gradients of twisted dynamic quantum gradients induced by intertwining sets of Clifford generators. Equation \ref{EQ.BSP.QOT_Type_II_Twisted_4} holds and implies strictly positive lower Ricci bounds. We therefore obtain arbitrary lower Ricci bounds by rescaling this example.
\end{bsp}

\begin{rem}\label{REM.QOT_Type_II_Twisted}
Example \ref{BSP.QOT_Type_II_Twisted} for $H=\mathbb{C}$, $\dim_{\mathbb{C}}\HII<\infty$, and fixed $C=\frac{1}{4}$ is introduced in \cite{ART.Car_Maa.2014.Quantum_OT_I}. For all $j\in\mathbb{N}$, note $\rho\lc{}v_{m+j+1}\rc\AII\lc\HII_{j}\rc\subset\AII\lc\HII_{j}\rc^{\perp}\subset L^{2}\lc\AII(\HII),\tau\rc$. If $m=\infty$, then we cannot use $\lset\partial_{n}\rset_{n\in\mathbb{N}}$ as noncommutative directional derivatives since locality is violated if we do not fix $m<\infty$ for Equation \ref{EQ.BSP.QOT_Type_II_Twisted_1} in Example \ref{BSP.QOT_Type_II_Twisted}.
\end{rem}


\subsubsection*{First and second quantisation of spectral triples}

Connes' program of noncommutative geometry \cite{BK.Con.1994.NCG}\cite{COL.Con.2021.NCG_Spectral_POV}\cite{BK.Kha.2013.NCG_Basic}\cite{BK.Kha_Mar.2008.NCG_Invitation} unifies continuous and discrete geometries \cite{BK.Gra_Var_Fig.2001.NCG_Elements}\cite{BK.vSui.2015.NCG_AF_Particle_Physics}\cite{BK.Var.2006.NCG_Elements_Short} using operator theory \cite{BK.Bla.2006.OpAlg}\cite{BK.Tak.1979.OpAlg_I}\cite{BK.Tak.2003.OpAlg_II}\cite{BK.Tak.2003.OpAlg_III}. His spectral reconstruction theorem shows commutative real spectral triples are operator algebraic formulation of compact spin geometry \cite{ART.Con.1996.NCG_Reconstruction}. All real spectral triples define noncommutative gauge theories \cite{ART.Cha_Con.1996.NCG_Spectral_Action_I}\cite{BK.vSui.2015.NCG_AF_Particle_Physics}\cite{BK.Var.2006.NCG_Elements_Short}. Inner fluctuations of noncommutative Dirac operators \cite{ART.Cha_Con.1996.NCG_Spectral_Action_I}\cite{ART.Cha_Con_vSui.2013.NCG_Inner_Fluctuations}\linebreak\cite{ART.Cha_Con_vSui.2020.NCG_Second_Quantisation}\cite{BK.vSui.2015.NCG_AF_Particle_Physics}\cite{BK.Var.2006.NCG_Elements_Short}, the latter being given as part of real spectral triples, determine a spectral action on gauge fields \cite{ART.Cha_Con.1996.NCG_Spectral_Action_I}\cite{ART.Cha_Con.1997.NCG_Spectral_Action_II}\cite{ART.Cha_Con_Mar.2007.NCG_Standard_Model_Recovered}. Following the spectral action principle of Connes and Chamseddine \cite{ART.Cha_Con.1997.NCG_Spectral_Action_II}, it is used as action functional driving the dynamics of bosonic gauge fields \cite{ART.Cha_Con.1996.NCG_Spectral_Action_I}\cite{BK.vSui.2015.NCG_AF_Particle_Physics}\cite{BK.Var.2006.NCG_Elements_Short}. This spectral paradigm derives the Standard Model of particle physics \cite{ART.Gai_Gra_Sci.1999.The_Standard_Model} from almost commutative geometries \cite{ART.Cha_Con_Mar.2007.NCG_Standard_Model_Recovered}, i.e.~mixed continuous-discrete noncommutative geometries. We review noncommutative gauge theory, give first and second quantisation of spectral triples, and outline how the latter yields our ansatz to study noncommutative gauge theories based on a proposed internalised spectral action if we generalise to quantum optimal transport parametrised by gauge fields.\par
We review noncommutative gauge theory. All spectral triples $\lc\mathfrak{A},H,D\rc$ consist of a unital pre-$C^{*}$-algebra $\mathfrak{A}$, faithful unital $^{*}$-representation $\pi:\mathfrak{A}\longrightarrow\BII(H)$ over separable Hilbert space $H$, as well as $D\in\UBII(H)_{h}$ with compact resolvent \lc{}cf.~Definition 4.30 in \cite{BK.vSui.2015.NCG_AF_Particle_Physics}\rc{}. Moreover, note $D$ satisfies properties showing it is a noncommutative analogue of an Atiyah–Singer–Dirac operator. We say that $\lc\mathfrak{A},H,D\rc$ is a real spectral triple if it is further equipped with real structure $J$ on $H$ intertwining with $D$ s.t.~the commutant property and first-order condition

\begin{align}\label{EQ.SSEC.QOT_DT_BSP_11}
\lb\pi(x),J\pi(y)^{*}J^{-1}\rb{}=0,\ \lb\overline{D\pi(x)-\pi(x)D},J\pi(y)^{*}J^{-1}\rb{}=0    
\end{align}

\noindent are satisfied for all $x,y\in\mathfrak{A}$ \lc{}cf.~Equation 4.3.1 in \cite{BK.vSui.2015.NCG_AF_Particle_Physics}\rc{}. The first-order condition as per Equation \ref{EQ.SSEC.QOT_DT_BSP_11} is an operator algebraic characterisation of $D$ as differential operator of order one. We ignore gradient operators here as they only signify even or odd dimension. We may disregard the first-order condition \cite{ART.Cha_Con_vSui.2013.NCG_Inner_Fluctuations} but do not do so here. First quantisation of compact spin manifolds as per Example \ref{BSP.QOT_First_Quantisation} clarifies the above analogies as it gives all commutative real spectral triples \cite{ART.Con.1996.NCG_Reconstruction}. Equation \ref{EQ.SSEC.QOT_DT_BSP_16} implies triviality of gauge groups in this case. We see general real spectral triples are necessary to describe abelian and non-abelian gauge theories \cite{ART.Gra_Mar_Var.1998.NCG_Standard_Model_Noncommutative_Necessary}. Second quantisation of spectral triples as per Example \ref{BSP.QOT_Second_Quantisation} yields description of the spectral action as per Equation \ref{EQ.SSEC.QOT_DT_BSP_15} in terms of quantum statistical mechanics \cite{BK.Bra.1987.OpAlg_Quantum_StM_I}\cite{BK.Bra.1987.OpAlg_Quantum_StM_II} as per Equation \ref{EQ.BSP.QOT_Second_Quantisation_6} using quantum relative entropy as per Definition \ref{DFN.Rel_Ent_AF_Cstar_Trace}. Two essential results in Example \ref{BSP.QOT_Second_Quantisation} are taken from \cite{ART.Cha_Con_vSui.2020.NCG_Second_Quantisation}. It leads us to formulate the internalised spectral action as per Equation \ref{EQ.BSP.QOT_Second_Quantisation_Parametrised_9}.\par


\pagebreak


To this end, we summarise relevant parts of noncommutative gauge theories defined by real spectral triples satisfying the first-order condition. Let $\lc\mathfrak{A},H,D,J\rc$ be such a real spectral triple. Norm closure of $\mathfrak{A}$ generates unital $C^{*}$-algebra $A$ s.t.~$\pi:A\longrightarrow\BII(H)$ is faithful unital $^{*}$-representation. Its $\mathfrak{A}$-bimodule of differential one-forms is defined by

\begin{align}\label{EQ.SSEC.QOT_DT_BSP_12}
\Omega_{D}^{1}\lc\mathfrak{A}\rc{}:=\lset{}T\in\BII(H)\ \vset\ \exists\lc{}x_{k},y_{k}\rc_{k=1}^{n}\subset\mathfrak{A}\times\mathfrak{A}:\ T=\sum_{k=1}^{n}\pi(x_{k})\cdot \overline{D\pi\lc{}y_{k}\rc{}-\pi\lc{}y_{k}\rc{}D}\rset{}
\end{align}

\noindent \lc{}cf.~Definition 4.36 in \cite{BK.vSui.2015.NCG_AF_Particle_Physics}\rc{}. Closure of unbounded commutators in Equation \ref{EQ.SSEC.QOT_DT_BSP_12} is ensured by the axioms of spectral triples. Moreover, get $\epsilon\in\lset\hspace{-0.025cm} \pm 1\rset$ s.t.~$JDJ^{-1}=\epsilon D$. For all hermitian connections $\nabla:\mathfrak{A}\longrightarrow\Omega_{D}^{1}\lc\mathfrak{A}\rc$, Theorem 6.15 and Theorem 6.16 in \cite{BK.vSui.2015.NCG_AF_Particle_Physics} imply the inner fluctuation of $D$ defined by

\begin{align}\label{EQ.SSEC.QOT_DT_BSP_13}
D_{T}:=D+T+\epsilon JTJ^{-1}
\end{align}

\noindent with $T:=\nabla 1_{A}\in\Omega_{D}^{1}\lc\mathfrak{A}\rc\cap\BII(H)_{h}$ yields real spectral triple $\lc\mathfrak{A},H,D_{T},J\rc$ \lc{}cf.~pp.112-114 in \cite{BK.vSui.2015.NCG_AF_Particle_Physics}\rc{}. We say that $T\in\Omega_{D}^{1}\lc\mathfrak{A}\rc\cap\BII(H)_{h}$ is a gauge field in this case. Proposition 6 in \cite{ART.Cha_Con_vSui.2013.NCG_Inner_Fluctuations} shows we have gauge semigroup

\begin{align}\label{EQ.SSEC.QOT_DT_BSP_14}
\Inn\lc\mathfrak{A},H,D\rc{}:=\lset{}T\in\Omega_{D}^{1}\lc\mathfrak{A}\rc\cap\BII(H)_{h}\ \vset\ T\ \textrm{is a gauge field}\rset{}
\end{align}

\noindent of $\lc\mathfrak{A},H,D,J\rc$. Its semigroup structure is not relevant to us. The map $T\mapsto D_{T}$ defined on $\Inn\lc\mathfrak{A},H,D\rc$ is a deformation of noncommutative Dirac operators parametrised by gauge fields. Assuming even $h:\mathbb{R}\longrightarrow [0,\infty)$ s.t.~finite trace is ensured in Equation \ref{EQ.SSEC.QOT_DT_BSP_15} below, the spectral action $S_{b}:\Inn\lc\mathfrak{A},H,D\rc\longrightarrow\mathbb{R}$ is defined by

\begin{align}\label{EQ.SSEC.QOT_DT_BSP_15}
S_{b}(T):=\tr\big(h\lc{}D_{T}\rc\big)
\end{align}

\noindent for all $T\in\Inn\lc\mathfrak{A},H,D\rc$ \lc{}cf.~Definition 7.1 in \cite{BK.vSui.2015.NCG_AF_Particle_Physics}\rc{}. We give suitable $h$ for our purposes in Example \ref{BSP.QOT_Second_Quantisation}. The subscript of $S_{b}$ denotes its use as action functional driving the dynamics of bosonic gauge fields \lc{}cf.~Theorem 11.10 in \cite{BK.vSui.2015.NCG_AF_Particle_Physics}\rc{}. There exist alternatives for other gauge fields, e.g.~the fermionic action \lc{}cf.~Definition 7.1 in \cite{BK.vSui.2015.NCG_AF_Particle_Physics}\rc{}.\par
The spectral action is a spectral invariant of $\lc\mathfrak{A},H,D,J\rc$. Proposition 6.17 in \cite{BK.vSui.2015.NCG_AF_Particle_Physics} shows each unitary Morita self-equivalences of $\lc\mathfrak{A},H,D,J\rc$ is implemented by a unique $U=\pi(u)J\pi(u)J^{-1}\in\UII(H)$ for $u\in\UII\lc\mathfrak{A}\rc$ s.t.~$T_{U}=\pi(u)\overline{D\pi(u)-\pi(u)D}\in\Inn\lc\mathfrak{A},H,D\rc$ is a gauge field. Proposition 6.5 in \cite{BK.vSui.2015.NCG_AF_Particle_Physics} shows we have gauge group

\begin{align}\label{EQ.SSEC.QOT_DT_BSP_16}
\mathfrak{G}\lc\mathfrak{A},H,D\rc{}:=\lset{}U\in\UII(H)\ \vset\ \exists u\in\UII\lc\mathfrak{A}\rc{}:\ U=\pi(u)J\pi(u)J^{-1}\rset\cong\UII\lc\mathfrak{A}\rc\big /\UII\lc\mathfrak{A}_{J}\rc{}
\end{align}

\noindent of $\lc\mathfrak{A},H,D,J\rc$ \lc{}cf.~Definition 6.4 in \cite{BK.vSui.2015.NCG_AF_Particle_Physics}\rc{}. Note $\UII(A_{j})\vartriangleleft\UII(A)$ as for Equation \ref{EQ.SSEC.QOT_DT_BSP_16} since $A_{J}=\lset{}x\in A\ \vset\ \pi(x)J=J^{*}\pi(x)\rset\subset Z(A)$. For all $U\in \mathfrak{G}\lc\mathfrak{A},H,D\rc$, we have $D_{T_{U}}=UDU^{*}$ by the first-order condition. We therefore see Equation \ref{EQ.SSEC.QOT_DT_BSP_15} is invariant under gauge transformations \lc{}cf.~Lemma \ref{LEM.FC_Unitary_Com} and Lemma \ref{LEM.FC_Preservation_II}\rc{}.\par


\pagebreak


We give first and second quantisation of spectral triples. Example \ref{BSP.QOT_First_Quantisation} gives first quantisation of compact spin manifolds \cite{ART.Con.1996.NCG_Reconstruction}. We further include general spectral triples as their own first quantisation by convention. Example \ref{BSP.QOT_Second_Quantisation} gives second quantisation of spectral triples \cite{ART.Cha_Con_vSui.2020.NCG_Second_Quantisation}. Remark \ref{REM.QOT_First_Second_Quantisation} briefly reviews the terminology of first and second quantisation as used in our discussion. Both underlying fundamental example classes use non-twisted dynamic quantum gradients arising from weak, equivalently norm, differentiation of trace-preserving $C^{*}$-dynamical systems determined by noncommutative Dirac operators, i.e.~assumes fixed gauge field. Example \ref{BSP.QOT_First_Quantisation} and Example \ref{BSP.QOT_Second_Quantisation} give quantum optimal transport without considering spatial coordinates. Upon passing to second quantisation, we introduce gauge fields as spatial coordinates. Example \ref{BSP.QOT_Second_Quantisation_Parametrised} generalises to quantum optimal transport parametrised by gauge fields via deforming noncommutative Dirac operators as per Equation \ref{EQ.SSEC.QOT_DT_BSP_13}.\par
We assume fixed gauge field and summarise results. First, Example \ref{BSP.QOT_First_Quantisation} arises from a conjugation group which splits into a spatial and quantum component as per Equation \ref{EQ.BSP.QOT_First_Quantisation_4}. We see quantum optimal transport is transversal to spatial optimal transport in this case. Equation \ref{EQ.BSP.QOT_First_Quantisation_2} shows the spatial component is generated by a quantisation of the gradient w.r.t.~the given Riemannian metric using the Clifford action \cite{BK.vSui.2015.NCG_AF_Particle_Physics}\cite{BK.Var.2006.NCG_Elements_Short}. Secondly, Example \ref{BSP.QOT_Second_Quantisation} rectifies transversality by quantising all spatial coordinates as per Equation \ref{EQ.BSP.QOT_Second_Quantisation_2}. We instead have a description in terms of quantum statistical mechanics without considering spatial coordinates. Equation \ref{EQ.BSP.QOT_Second_Quantisation_2} gives an explicit formula for Equation \ref{EQ.BSP.QOT_Type_II_5}. The formula is taken from Proposition 2.6 in \cite{ART.Cha_Con_vSui.2020.NCG_Second_Quantisation}. However, note we explicitly solve the associated implementation problem \cite{BK.Ply_Rob.1994.Clifford_Algebras} in Lemma \ref{LEM.QOT_Second_Quantisation_Implementation_Fock}. Assuming trace-class, Equation \ref{EQ.BSP.QOT_Second_Quantisation_2} shows the given non-twisted dynamic quantum gradient is infinitesimal evolution of observables in wedged Heisenberg representation at thermal equilibrium determined by a KMS-state \cite{BK.Bra.1987.OpAlg_Quantum_StM_II} of the given trace-preserving local $C^{*}$-dynamical system. Up to sign, Corollary \ref{COR.Wstar_Derivation_QG_Dynamic_System} shows such description transfers to quantum Laplacians by twice application. We therefore expect properties of quantum optimal transport as stated in the introduction of this chapter.

\begin{bsp}\label{BSP.QOT_First_Quantisation}
We assume commutative real spectral triples, i.e.~first quantisation of compact spin manifolds \cite{ART.Con.1996.NCG_Reconstruction}. Let $\lc{}X,g\rc$ be a compact spin manifold, $S\longrightarrow X$ its spinor bundle and $D$ its Atiyah–Singer–Dirac operator \cite{ART.Con.1996.NCG_Reconstruction}\cite{BK.Tha.1992.The_Dirac_Equation}\cite{BK.vSui.2015.NCG_AF_Particle_Physics}\cite{BK.Var.2006.NCG_Elements_Short}.\par
We have unital pre-$C^{*}$-algebra $C^{\infty}(X)$ and $C^{*}$-algebra $C(X)$ \lc{}cf.~Example \ref{BSP.Cstar_Commutative}\rc{}. For all $x\in X$, the finite-dimensional Clifford algebra $S_{x}=\AII\lc{}T_{x}^{*}X\rc$ has inner product induced by the cotangent Riemannian metric. We extend pointwise left-multiplication of scalars from $C^{\infty}\lc{}X,T^{*}X\rc$ to $L^{2}\lc{}X,S\rc$. This defines faithful unital $^{*}$-representation $L:C(X)\longrightarrow\BII\lc{}L^{2}\lc{}X,S\rc\rc$. Fibrewise right-multiplication of elements in Clifford algebras defines Clifford action $c:C^{\infty}\lc{}X,T^{*}X\rc\longrightarrow\BII\lc{}L^{2}\lc{}X,S\rc\rc$. Up to $L_{-i}$, $D$ is the concatenation of $c$ and the spin connection of $\lc{}X,g\rc$, i.e.~the unique lift of the Levi-Civita connection associated to $\lc{}X,g\rc$ from $T^{*}X$ to $S$. The charge conjugation $J_{X}$ of $S$ is a suitable real structure on $L^{2}\lc{}X,S\rc$. Altogether, we construct the canonical commutative real spectral triple $\lc\mathfrak{A},H,D,J\rc{}=\lc{}C^{\infty}(X),L^{2}\lc{}X,S\rc{},D,J_{X}\rc$ \cite{ART.Con.1996.NCG_Reconstruction}. We assume the latter without loss of generality. For details on the construction and our application of its properties, we refer to Chapter 4 in \cite{BK.vSui.2015.NCG_AF_Particle_Physics} and Chapter 3 in \cite{BK.Var.2006.NCG_Elements_Short}.\par
We know spectra of elements in $\KII\lc{}L^{2}\lc{}X,S\rc\rc_{h}$ are discrete by the spectral theorem for self-adjoint unbounded operators \lc{}cf.~Theorem 5.7 in \cite{BK.Sch.2012.Unbounded_Operators}\rc{}. Continuity of elements in $C(X)$ implies $L\lc{}C(X)\rc\cap\KII(L^{2}(X,S))=0$ by the intermediate value theorem as spectra of continuous functions on $X$ are subsets of their images by compactness.\par
We claim the conjugation group $\Ad^{D}:\mathbb{R}\longrightarrow\mathrm{Aut}\lc\BII\lc{}L^{2}\lc{}X,S\rc\rc\rc$ of $D$ given by

\begin{align}\label{EQ.BSP.QOT_First_Quantisation_1}
\textrm{Ad}_{t}^{D}(S)=e^{itD}Se^{-itD}
\end{align}

\noindent for all $t\in\mathbb{R}$ and $S\in\BII\lc{}L^{2}\lc{}X,S\rc\rc$ splits into a spatial and quantum component as per Equation \ref{EQ.BSP.QOT_First_Quantisation_4} upon restriction to $L\lc{}C(X)\rc\oplus\KII\lc{}L^{2}\lc{}X,S\rc\rc\subset\BII\lc{}L^{2}\lc{}X,S\rc\rc$. We consider their generators. If $h\in C^{\infty}(X)$, then the two conditions for Equation \ref{EQ.SSEC.NCDS_NCG_QG_9} are met for $H=L^{2}\lc{}X,S\rc$, $\DII=D$ and $S=L_{h}$. As per Theorem 4.20 in \cite{BK.vSui.2015.NCG_AF_Particle_Physics} and explained on pp.8-10 in \cite{BK.Var.2006.NCG_Elements_Short}, we obtain

\begin{align}\label{EQ.BSP.QOT_First_Quantisation_2}
\nabla^{\textrm{spt}}h:=\restr{0.925}{\frac{d}{dt}}{t=0,\w}\textrm{Ad}_{t}^{D}\lc{}L_{h}\rc{}=\overline{i\lc{}DL_{h}-L_{h}D\rc{}}=-ic\lc{}dh\rc{}
\end{align}

\noindent for all $h\in C^{\infty}(X)$. Smoothness and Equation \ref{EQ.BSP.QOT_First_Quantisation_2} imply $-ic\lc{}dh\rc\in L\lc{}C^{\infty}(X)\rc$ in each case by the first-order property \lc{}cf.~Equation 4.3.1 in \cite{BK.vSui.2015.NCG_AF_Particle_Physics}\rc{}. As such, Equation \ref{EQ.BSP.QOT_First_Quantisation_2} integrates to $\restr{0.925}{\Ad^{D}}{L\lc{}C(X)\rc{}}:\mathbb{R}\longrightarrow\mathrm{Aut}\lc{}L(X)\rc$. Applying constructions in Example \ref{BSP.QOT_Type_I} to $H=L^{2}\lc{}X,S\rc$ and $D\in\UBII\lc{}L^{2}\lc{}X,S\rc\rc$, we see Equation \ref{EQ.BSP.QOT_Type_I_8} yields non-twisted dynamic quantum gradient $\nabla^{\textrm{qtm}}:\KII\lc{}L^{2}\lc{}X,S\rc\rc_{0}\longrightarrow S^{2}\lc{}L^{2}\lc{}X,S\rc\rc$ given by 

\begin{align}\label{EQ.BSP.QOT_First_Quantisation_3}
\nabla^{\textrm{qtm}}x:=\nabla^{D}x=\restr{0.925}{\frac{d}{dt}}{t=0,\w}\textrm{Ad}_{t}^{D}(T)=\overline{i\lc{}Dx-xD\rc{}}
\end{align}

\noindent for all $x\in\KII\lc{}L^{2}\lc{}X,S\rc\rc_{0}$. Note $\overline{i\lc{}Dx-xD\rc{}}\in\KII\lc{}L^{2}\lc{}X,S\rc\rc_{0}$ in each case by construction. As such, Equation \ref{EQ.BSP.QOT_First_Quantisation_3} integrates to $\restr{0.925}{\Ad^{D}}{\KII(L^{2}(X,S))}:\mathbb{R}\longrightarrow\mathrm{Aut}\lc\KII\lc{}L^{2}\lc{}X,S\rc\rc\rc$.\par
Using $L\lc{}C(X)\rc\cap\KII(L^{2}(X,S))=0$, note Equation \ref{EQ.BSP.QOT_First_Quantisation_2} and Equation \ref{EQ.BSP.QOT_First_Quantisation_3} together integrate to $\restr{0.925}{\Ad^{D}}{L\lc{}C(X)\rc\oplus\KII(L^{2}(X,S))}:\mathbb{R}\longrightarrow\mathrm{Aut}\lc{}L\lc{}C(X)\rc\oplus\KII\lc{}L^{2}\lc{}X,S\rc\rc\rc$ given by

\begin{align}\label{EQ.BSP.QOT_First_Quantisation_4}
\restr{0.925}{\textrm{Ad}_{t}^{D}}{L\lc{}C(X)\rc\oplus\KII(L^{2}(X,S))}=\restr{0.925}{\textrm{Ad}_{t}^{D}}{L\lc{}C(X)\rc{}}\oplus\restr{0.925}{\textrm{Ad}_{t}^{D}}{\KII(L^{2}(X,S))}
\end{align}

\noindent for all $t\in\mathbb{R}$. Note Equation \ref{EQ.BSP.QOT_First_Quantisation_2} and Equation \ref{EQ.BSP.QOT_First_Quantisation_3} are infinitesimal evolution of observables in Heisenberg representation \lc{}cf.~pp.3-15 in \cite{BK.Bra.1987.OpAlg_Quantum_StM_I}\rc{}. They are first quantisations in the sense of Remark \ref{REM.QOT_First_Second_Quantisation}. Equation \ref{EQ.BSP.QOT_First_Quantisation_2} in fact quantises the gradient on $X$ w.r.t.~$g$ using the Clifford action. We say that $\restr{0.925}{\Ad^{D}}{L\lc{}C(X)\rc{}}$ is the spatial, and $\restr{0.925}{\Ad^{D}}{\KII(L^{2}(X,S))}$ the quantum component of the time-evolution $\Ad^{D}$ of observables.\par
Example \ref{BSP.QOT_Type_I} shows the quantum gradient $\nabla^{\textrm{qtm}}=\nabla^{D}$ lets us define quantum optimal transport for the tracial AF-$C^{*}$-algebra $\lc\KII(L^{2}(X,S)),\tr\rc$. Assuming the spatial gradient $\nabla^{\textrm{spt}}$ yields a notion of spatial optimal transport for the tracial $C^{*}$-algebra $\lc{}C(X),\int d\vol\rc$, e.g.~as per \cite{PRE.Wir.2018.NC_OT}, we see quantum optimal transport is transversal to spatial optimal transport by the direct sum decomposition in Equation \ref{EQ.BSP.QOT_First_Quantisation_4}.\par


\pagebreak


Spatiality of $C(X)$, resp.~$L\lc{}C(X)\rc$, is obvious as spatial coordinates parametrise the Riemannian manifold $X$ and therefore observables given by elements in $L\lc{}C(X)\rc$. We consider $L\lc{}C(X)\rc\subset\BII\lc{}L^{2}\lc{}X,S\rc\rc$ to formulate a necessary condition for spatiality here. We use quantum relative entropy as per Definition \ref{DFN.Rel_Ent_AF_Cstar_Trace}. We apply said condition in Example \ref{BSP.QOT_Second_Quantisation} to argue second quantisation quantises all, ergo considers no, spatial coordinates. For all $\mu\in\SII\lc\KII\lc{}L^{2}\lc{}X,S\rc\rc$, $\Ent\lc\mu,\tr\rc\in [-\infty,\infty]$ is the relative entropy of $\mu$ w.r.t.~$\tr$ as per Equation \ref{EQ.DFN.Rel_Ent_AF_Cstar_Trace_1}. Theorem \ref{THM.Rel_Ent_AF_Cstar_Trace} ensures it measures information required to discriminate $\mu$ and $\tr$ through observation by extending its use from the strongly unital finite-trace case \lc{}cf.~pp.1-11 in \cite{BK.Ohy_Pet.1993.Rel_Ent}\rc{}. If $h\in C(X)_{+}$, then Lemma \ref{LEM.Rel_Ent_AF_Cstar_Trace_II} shows there exists no weakly dense subset $K\subset \KII\lc{}L^{2}\lc{}X,S\rc\rc\cap S^{1}\lc{}L^{2}\lc{}X,S\rc\rc{}=S^{1}\lc{}L^{2}\lc{}X,S\rc\rc$ and $C>0$ s.t.~the map $\tilde{\mu}_{h}:K\longrightarrow\mathbb{C}$ defined by

\begin{align}\label{EQ.BSP.QOT_First_Quantisation_5}
\tilde{\mu}_{h}(T):=C^{-1}\tr\lc{}L_{x}T\rc{}
\end{align}

\noindent for all $T\in K$ extends to a $\mu_{h}\in\SII\lc\KII\lc{}L^{2}\lc{}X,S\rc\rc\rc$ with

\begin{align}\label{EQ.BSP.QOT_First_Quantisation_6}
\babsv{1.15}{\Ent\lc\mu_{h},\tr\rc{}}<\infty.
\end{align}

\noindent Indeed, $1)$ in Lemma \ref{LEM.Rel_Ent_AF_Cstar_Trace_II} shows Equation \ref{EQ.BSP.QOT_First_Quantisation_6} implies $h\in L\lc{}C(X)\rc\cap S^{1}\lc{}L^{2}\lc{}X,S\rc\rc{}=0$ since $S^{1}\lc{}L^{2}\lc{}X,S\rc\rc\subset\KII\lc{}L^{2}\lc{}X,S\rc$ \lc{}cf.~Example \ref{BSP.Wstar}\rc{}. Assuming hyperfinite factor, our necessary condition for spatiality is non-extension w.r.t.~the canonical trace.\par
We motivate our condition. The volume form $d\vol$ is a non-atomic Radon measure on $X$ \lc{}cf.~pp.299-306 in \cite{BK.Lan.1995.Riemannian_Manifolds}\rc{}. Non-extension implies our measuring process of quantum information as difference of observables from quantum white noise, up to musical isomorphisms, fails for all positive continuous ones parametrised by spatial coordinates. Corollary \ref{COR.Rel_Ent_AF_Cstar_Trace} shows said process, in contrast to any associated to relative entropy w.r.t.~$\int d\vol$, only considers differences on discrete spectra. Naturally, such a countable process cannot discern observables as above by the intermediate value theorem. We see our measuring process fails since it requires us to measure with absolute precision \cite{ART.Dav_Lew.1970.Wstar_Quantum_Probability} and this is prevented by infinitesimal length elements \cite{BK.Con.1994.NCG}\cite{BK.Lan.1995.Riemannian_Manifolds}.
\end{bsp}

\begin{bsp}\label{BSP.QOT_Second_Quantisation}
Let $H$ be a separable Hilbert space and $D\in\UBII(H)_{h}$ with compact resolvent, e.g.~given by a real spectral triple. We use finite-dimensional approximation $\lset{}H_{j}\rset_{j\in\mathbb{N}}$ of $H$ via eigenvectors as per Example \ref{BSP.QOT_Type_I}. The setting of Example \ref{BSP.QOT_Type_II_1} requires orthogonal complex structure $J$ on $H$ s.t.~it is $H_{j}$-reducible for all $j\in\mathbb{N}$. We use the one in \cite{ART.Cha_Con_vSui.2020.NCG_Second_Quantisation}. Let $P_{\pm}:H\longrightarrow E_{\pm}$ be Hilbert space projections onto the eigenvectors of $D$ with non-negative, resp.~non-positive eigenvalues. Set $J:=i\lc{}P_{+}-P_{-}\rc$. We directly verify $J$ is orthogonal complex structure on $H$ s.t.~

\begin{align}\label{EQ.BSP.QOT_Second_Quantisation_1}
DJ=JD.
\end{align}

\noindent Equation \ref{EQ.BSP.QOT_Second_Quantisation_1} shows we are in the setting of Example \ref{BSP.QOT_Type_II_1}. The second quantisation map $T\mapsto\bigwedge T$ from $\BII(H)$ to $\BII\lc\FII(H[J])\rc$ exists \lc{}cf.~pp.6-10 in \cite{BK.Bra.1987.OpAlg_Quantum_StM_II}\rc{}. Example \ref{BSP.QOT_Type_II_1} also gives $\tau$-preserving local $C^{*}$-dynamical system $\lc\AII(H[J]),\mathbb{R},\Cliff_{J}\lc{}e^{itD}\rc\rc$.\par
For all $t\in\mathbb{R}$ and $x\in\AII(H[J])$, Lemma \ref{LEM.QOT_Second_Quantisation_Implementation_Fock} shows

\begin{align}\label{EQ.BSP.QOT_Second_Quantisation_2}
\textrm{Cliff}_{J}\lc{}e^{itD}\rc{}(x)=\bigwedge e^{it\absv{1.15}{D}}x\bigwedge e^{-it\absv{1.15}{D}}\in\AII(H[J]).
\end{align}

\noindent Equation \ref{EQ.BSP.QOT_Second_Quantisation_2} is the claimed explicit formula for Equation \ref{EQ.BSP.QOT_Type_II_5}. Passing from $D$ to $\absv{1.15}{D}$ in the second quantisation map avoids negative eigenvalues, i.e.~the Dirac sea \cite{ART.Cha_Con_vSui.2020.NCG_Second_Quantisation}\cite{BK.Tha.1992.The_Dirac_Equation}. We prove Lemma \ref{LEM.QOT_Second_Quantisation_Implementation_Fock} by solving the associated implementation problem \cite{BK.Ply_Rob.1994.Clifford_Algebras}. We say that $\Cliff_{J}\lc{}e^{itD}\rc$ is implemented on $\FII(H[J])$ by $\bigwedge e^{it\absv{1.15}{D}}$ in each case.\par
Following Equation \ref{EQ.BSP.QOT_Second_Quantisation_2}, Equation \ref{EQ.BSP.QOT_Second_Quantisation_3} links quantum optimal transport and quantum statistical mechanics \cite{BK.Bra.1987.OpAlg_Quantum_StM_I}\cite{BK.Bra.1987.OpAlg_Quantum_StM_II}. We use KMS-states below \lc{}cf.~Definition 5.3.1 in \cite{BK.Bra.1987.OpAlg_Quantum_StM_II}\rc{}. If $e^{-\beta\absv{1.15}{D}}\in S^{1}(H)$ for given inverse temperature $\beta\in\mathbb{R}$, then Proposition 2.6 in \cite{ART.Cha_Con_vSui.2020.NCG_Second_Quantisation} specifies Example 5.3.2 in \cite{BK.Bra.1987.OpAlg_Quantum_StM_II} by showing the unique KMS$_{\beta}$-state of the $\tau$-preserving local $C^{*}$-dynamical system $\lc\AII(H[J]),\mathbb{R},\Cliff_{J}\lc{}e^{itD}\rc\rc$ has density operator 

\begin{align}\label{EQ.BSP.QOT_Second_Quantisation_3}
P_{D}:=\tr\lc\bigwedge e^{-\beta\absv{1.15}{D}}\rc^{-1}\cdot \bigwedge e^{-\beta\absv{1.15}{D}}\in S^{1}\lc\FII(H[J])\rc_{+}.
\end{align}

\noindent Applying constructions in Example \ref{BSP.QOT_Type_II_1} to $\alpha_{t}=\Cliff_{J}\lc{}e^{itD}\rc$, Corollary \ref{COR.Wstar_Derivation_QG_Dynamic_System} yields non-twisted dynamic quantum gradient $\nabla^{\textrm{qtm}}:\AII(H[J])_{0}\longrightarrow L^{2}\lc\AII(H[J]),\tau\rc$ given by 

\begin{align}\label{EQ.BSP.QOT_Second_Quantisation_4}
\nabla^{\textrm{qtm}}x:=\nabla^{\DII_{\alpha}}(x)=\restr{0.925}{\frac{d}{dt}}{t=0,\w}\bigwedge e^{it\absv{1.15}{D}}x\bigwedge e^{-it\absv{1.15}{D}}
\end{align}

\noindent for all $x\in\AII(H[J])_{0}$. Note Equation \ref{EQ.BSP.QOT_Second_Quantisation_3} then implies Equation \ref{EQ.BSP.QOT_Second_Quantisation_4} is infinitesimal evolution of observables in wedged Heisenberg representation at thermal equilibrium determined by $P_{D}$. This is a second quantisation in the sense of Remark \ref{REM.QOT_First_Second_Quantisation}. Whereas Equation \ref{EQ.BSP.QOT_First_Quantisation_3} has closed form as unbounded commutator, use of the infinite exterior algebra on the right-hand side of Equation \ref{EQ.BSP.QOT_Second_Quantisation_4} introduces converging double sums with varying left-and right-multiplication of $\pm i\absv{1.15}{D}$ preventing a ready closed form.\par
Example \ref{BSP.QOT_Type_II_1} shows the quantum gradient $\nabla^{\textrm{qtm}}=\nabla^{\DII_{\alpha}}$ lets us define quantum optimal transport for the tracial AF-$C^{*}$-algebra $\lc\AII(H[J]),\tau\rc$. Following our discussion at the end of Example \ref{BSP.QOT_First_Quantisation}, $\tau<\infty$ implies our necessary condition for spatiality is not satisfied. For all $\mu\in\SII\lc\AII(H[J])\rc$, $\Ent(\mu,\tau)\in [-\infty,\infty]$ is the relative entropy of $\mu$ w.r.t.~$\tau$ as per Equation \ref{EQ.DFN.Rel_Ent_AF_Cstar_Trace_1}. We know $\AII(H[J])\subset L^{1}\lc\AII(H[J]),\tau\rc$ is weakly dense since $\tau<\infty$ \lc{}cf.~Proposition \ref{PRP.Wstar_NCI_VIII}\rc{}. For all $x\in \AII(H[J])_{+}$, Corollary \ref{COR.Rel_Ent_AF_Cstar_Trace} for $p=1_{A}$ shows 

\begin{align}\label{EQ.BSP.QOT_Second_Quantisation_5}
\mu_{x}:=\tau(x)^{-1}x^{\flat}\in\mathcal{S}^{\NI}\lc\AII(H[J])\rc{}
\end{align}

\noindent as per Equation \ref{EQ.BSP.QOT_First_Quantisation_5} for $K=\AII(H[J])$ has $\absv{1.15}{\Ent\lc\mu_{x},\tau\rc\hspace{0.025cm}}<\infty$. Our necessary condition is not satisfied. We see $\nabla^{\textrm{qtm}}$ quantises all, ergo considers no, spatial coordinates. Assuming commutative real spectral triple, $\nabla^{\textrm{qtm}}$ subsumes the generators of both components on the right-hand side of Equation \ref{EQ.BSP.QOT_First_Quantisation_4} because Equation \ref{EQ.BSP.QOT_Second_Quantisation_2} is a second quantisation of the unrestricted time-evolution as per Equation \ref{EQ.BSP.QOT_First_Quantisation_1}.\par
If $H$ and $D\in\UBII(H)_{h}$ are given by a real spectral triple, commutative or not, then we describe the spectral action as per Equation \ref{EQ.SSEC.QOT_DT_BSP_15} using the negative of quantum relative entropy w.r.t.~$\tr$, i.e.~von Neumann entropy \lc{}cf.~p.17 in \cite{BK.Ohy_Pet.1993.Rel_Ent}\rc{}. Let $T$ be the fixed gauge field, $D_{T}:=D$ and $P_{T}:=P_{D}$. For all $\lambda\in\mathbb{R}$, set $h(\lambda):=\log\lc{}1+e^{-\lambda}\rc{}+\lambda e^{-\lambda}\lc{}1+e^{-\lambda}\rc^{-1}$. Corollary 3.2 in \cite{ART.Cha_Con_vSui.2020.NCG_Second_Quantisation} implies $h:\mathbb{R}\longrightarrow [0,\infty)$ is even. Theorem 3.4 in \cite{ART.Cha_Con_vSui.2020.NCG_Second_Quantisation} shows

\begin{align}\label{EQ.BSP.QOT_Second_Quantisation_6}
S_{b}(T)=\tr\big(h\lc{}D_{T}\rc\big)=-\tr\big(P_{T}\log P_{T}\big)=-\Ent\hspace{-0.0375cm} \big(P_{T}^{\flat},\tr\big)<\infty.
\end{align}

\noindent Unfortunately, Equation \ref{EQ.BSP.QOT_Second_Quantisation_6} uses quantum relative entropy w.r.t.~$\tr$ and not $\tau$. We want the latter for an ansatz to study the dynamics of gauge fields driven by varying Equation \ref{EQ.BSP.QOT_Second_Quantisation_6} via deforming Equation \ref{EQ.SSEC.QOT_DT_BSP_13}. We propose to internalise Equation \ref{EQ.BSP.QOT_Second_Quantisation_6} as per Equation \ref{EQ.SSEC.QOT_DT_BSP_17} and generalise to Equation \ref{EQ.BSP.QOT_Second_Quantisation_Parametrised_9} in Example \ref{BSP.QOT_Second_Quantisation_Parametrised}. Note \cite{ART.Cha_Con_vSui.2020.NCG_Second_Quantisation} is based on \cite{ART.Cha_Con.1997.NCG_Spectral_Action_II}\cite{ART.Cha_Con_vSui.2013.NCG_Inner_Fluctuations}. We moreover refer to \cite{BK.vSui.2015.NCG_AF_Particle_Physics} as comprehensive treatment of the latter. The general noncommutative geometric approach to quantum thermodynamics used in \cite{ART.Cha_Con_vSui.2020.NCG_Second_Quantisation} is introduced and explained as part of \cite{ART.Con_Rov.1994.Wstar_Time_Thermodynamics}.
\end{bsp}

\begin{rem}\label{REM.QOT_First_Second_Quantisation}
First and second quantisation denotes, to our knowledge, Hamiltonian formalism for a single quantum system, resp.~multiple, often countable infinitely many interacting ones \lc{}cf.~pp.1-38 in \cite{BK.Ste_vLee.2013.Full_Quantum_StM}\rc{}. The latter arises as infinitely many copies of the former by applying to it the second quantisation map. If we consider time-evolution of fermions in Heisenberg representation \lc{}cf.~pp.3-15 in \cite{BK.Bra.1987.OpAlg_Quantum_StM_I} and pp.6-10 in \cite{BK.Bra.1987.OpAlg_Quantum_StM_II}\rc{}, then Example \ref{BSP.QOT_Second_Quantisation} indeed lifts time-evolution in Example \ref{BSP.QOT_First_Quantisation} as per Equation \ref{EQ.BSP.QOT_Second_Quantisation_2} by mapping both given constituent semigroups to their wedged form.
\end{rem}

Example \ref{BSP.QOT_Second_Quantisation_Parametrised} outlines how Example \ref{BSP.QOT_Second_Quantisation}, specifically Equation \ref{EQ.BSP.QOT_Second_Quantisation_6}, yields an ansatz to study noncommutative gauge theories based on the internalised spectral action as per Equation \ref{EQ.BSP.QOT_Second_Quantisation_Parametrised_9} if we generalise to quantum optimal transport parametrised by gauge fields. Let $\lc\mathfrak{A},H,D,J\rc$ be a real spectral triple. We suppress $J$ below as we use its symbol for orthogonal complex structures as per Example \ref{BSP.QOT_Second_Quantisation}. For all gauge fields $T\in\Inn\lc\mathfrak{A},H,D\rc$, we have $J_{T}$ as per Example \ref{BSP.QOT_Second_Quantisation} for $D_{T}$ as per Equation \ref{EQ.SSEC.QOT_DT_BSP_13}. If we further have a map $\Inn:\Inn\lc\mathfrak{A},H,D\rc\longrightarrow\SII\lc\AII(H)\rc$, then its associated internalisation of Equation \ref{EQ.BSP.QOT_Second_Quantisation_6} using quantum relative entropy w.r.t.~$\tau$ is given by

\begin{align}\label{EQ.SSEC.QOT_DT_BSP_17}
S_{b}^{\Inn}(T)=-\Ent\hspace{-0.0375cm} \bigg(\big(\rho_{J_{T}}^{-1}\big)^{*}\big(\Inn(T)\big),\tau\bigg)
\end{align}

\noindent for all $T\in\Inn\lc\mathfrak{A},H,D\rc$. Note Equation \ref{EQ.SSEC.QOT_DT_BSP_17} uses $\rho_{J_{T}}$ as per Equation \ref{EQ.SSEC.QOT_DT_BSP_9} in each case. We generalise Equation \ref{EQ.SSEC.QOT_DT_BSP_17} to Equation \ref{EQ.BSP.QOT_Second_Quantisation_Parametrised_9} in Example \ref{BSP.QOT_Second_Quantisation_Parametrised} by considering all normalised Radon measures on finite-dimensional spaces of admissible gauge fields evaluating in $\AII(H)$ up to varying $\rho_{J_{T}}$ as per Equation \ref{EQ.BSP.QOT_Second_Quantisation_Parametrised_1}, i.e.~states on continuous fields of AF-$C^{*}$-algebras. If key technical challenges are solved in future work, then we hope to study the dynamics of such generalised gauge fields described as gradient flows driven by the internalised spectral action for the given parametrised quantum optimal transport. We are motivated by the classical approach of Jordan, Kinderlehrer and Otto for Fokker-Planck equations \cite{ART.Jor_Kin_Ott.1998.Fokker_Planck}\cite{ART.Ott.2001.Classical_OT_Porous_Medium}\cite{ART.Ott.2005.Classical_OT_GradFlow_DisConvex}.

\begin{bsp}\label{BSP.QOT_Second_Quantisation_Parametrised}
Let $\lc\mathfrak{A},H,D,J\rc$ be a real spectral triple. We suppress $J$. For all gauge fields $T\in\Inn\lc\mathfrak{A},H,D\rc$, we have $J_{T}$ as per Example \ref{BSP.QOT_Second_Quantisation} for $D_{T}$ as per Equation \ref{EQ.SSEC.QOT_DT_BSP_13}. We do not know of a locally compact topology on $\Inn\lc\mathfrak{A},H,D\rc$ allowing for constructions as below. We instead consider $X\subset\Inn\lc\mathfrak{A},H,D\rc$ s.t.~four conditions are satisfied.\par
First, let $\lc{}X,g\rc$ be a smooth Riemannian manifold. We equip $T^{*}X\cong TX$ with its canonical dual Riemannian metric. Secondly, let $d\vol$ be a finite unoriented volume form, also called volume element, on $X$ \lc{}cf.~pp.299-306 in \cite{BK.Lan.1995.Riemannian_Manifolds}\rc{}. Thirdly, let $\beta:X\longrightarrow\mathbb{R}$ be smooth s.t.~$e^{-\beta(T)\absv{1.15}{D_{T}}}\in S^{1}(H)$ for all $T\in X$. Fourthly, let

\begin{align}\label{EQ.BSP.QOT_Second_Quantisation_Parametrised_1}
A_{X}:=\coprod_{T\in X}\AII\lc{}H\lb{}J_{T}\rb\rc{}=\coprod_{T\in X}\rho_{J_{T}}\lc\AII(H)\rc{}
\end{align}

\noindent determine both a smooth vector bundle and u.s.c.~$C^{*}$-bundle over $X$ \lc{}cf.~Definition $6.18$ in \cite{BK.vSui.2015.NCG_AF_Particle_Physics}\rc{}. Its space of continuous sections $\Gamma\lc{}A_{X}\rc$ is a $C^{*}$-algebra with norm given by $\dblv{}F\dblv_{\Gamma\lc{}A_{X}\rc{}}:=\sup_{T\in X}\dblv{}F(T)\dblv_{\AII\lc{}H\lb{}J_{T}\rb\rc{}}$ for all $F\in\Gamma\lc{}A_{X}\rc$ \lc{}cf.~Proposition 6.19 in \cite{BK.vSui.2015.NCG_AF_Particle_Physics}\rc{}. We define l.s.c.~faithful semi-finite trace $\int_{X}\tau d\vol$ on $\Gamma\lc{}A_{X}\rc$ by setting

\begin{align}\label{EQ.BSP.QOT_Second_Quantisation_Parametrised_2}
\lc\int_{X}\tau d\vol\rc(F):=\int_{X}\tau\big(F(T)\big)d\vol   
\end{align}

\noindent for all $F\in\Gamma\lc{}A_{X}\rc_{+}$. We have tracial $C^{*}$-algebra $\lc\Gamma\lc{}A_{X}\rc{},\int_{X}\tau d\vol\rc$ in the space of bounded measurable sections $L^{\infty}\lc\Gamma\lc{}A_{X}\rc{},\int_{X}\tau d\vol\rc$ \lc{}cf.~Proposition \ref{PRP.Wstar_Trace_Ext_I} and Proposition \ref{PRP.Wstar_Trace_Ext_III}\rc{}. Note $L^{2}\lc\Gamma\lc{}A_{X}\rc{},\int_{X}\tau d\vol\rc$ equipped with canonical left-~and right-actions and pointwise algebra involution is a symmetric $W^{*}$-bimodule over $L^{\infty}\lc\Gamma\lc{}A_{X}\rc{},\int_{X}\tau d\vol\rc$.\par
We define noncommutative gradient as per Equation \ref{EQ.BSP.QOT_Second_Quantisation_Parametrised_6} with domain

\begin{align}\label{EQ.BSP.QOT_Second_Quantisation_Parametrised_3}
\Gamma_{0}^{\infty}\lc{}A_{X}\rc{}:=\lset{}F\in\Gamma^{\infty}\lc{}A_{X}\rc\ \vset\ \forall T\in X:\ F(T)\in\AII\lc{}H\lb{}J_{T}\rb\rc_{0}\rset{}.
\end{align}

\noindent Equation \ref{EQ.BSP.QOT_Second_Quantisation_Parametrised_4} gives the fundamental compatibility condition for spatial and quantum components. The latter assumes continuous action of $\Gamma^{\infty}\lc{}T^{*}X\rc$ on $\Gamma^{\infty}\lc{}A_{X}\rc$ motivated by tensor contraction. We allow for loss of regularity. Let $\nabla^{X}:\Gamma^{\infty}\lc{}A_{X}\rc\longrightarrow\Gamma^{\infty}\lc{}T^{*}X\otimes A_{X}\rc$ be a covariant derivative and $\mathfrak{C}:\Gamma^{\infty}\lc{}T^{*}X\otimes A_{X}\rc\longrightarrow\Gamma\lc{}A_{X}\rc$ a bounded linear map. Assume we define symmetric $W^{*}$-derivation $\nabla^{h}:\Gamma_{0}^{\infty}\lc{}A_{X}\rc\longrightarrow L^{2}\lc\Gamma\lc{}A_{X}\rc{},\int_{X}\tau d\vol\rc$ by setting

\begin{align}\label{EQ.BSP.QOT_Second_Quantisation_Parametrised_4}
\big(\nabla^{h} F\big)(T):=\mathfrak{C}\lc\nabla^{X}F\rc{}(T)
\end{align}

\noindent for all $F\in\Gamma_{0}^{\infty}\lc{}A_{X}\rc$ and $T\in X$. We call it a spatial, or horizontal gradient. Applying constructions in Example \ref{BSP.QOT_Second_Quantisation} to $J_{T}$ and $D_{T}$ in each case, we see Equation \ref{EQ.BSP.QOT_Second_Quantisation_2} lets us define symmetric $W^{*}$-derivation $\nabla^{v}:\Gamma_{0}^{\infty}\lc{}A_{X}\rc\longrightarrow\Gamma\lc{}A_{X}\rc$ by setting 

\begin{align}\label{EQ.BSP.QOT_Second_Quantisation_Parametrised_5}
\big(\nabla^{v} F\big)(T):=\restr{0.925}{\frac{d}{dt}}{t=0,\w}\bigwedge e^{it\absv{1.15}{D_{T}}}F(T)\bigwedge e^{-it\absv{1.15}{D_{T}}}
\end{align}

\noindent for all $F\in\Gamma_{0}^{\infty}\lc{}A_{X}\rc$ and $T\in X$. We call it a total quantum, or vertical gradient.\par


\pagebreak


We define symmetric $W^{*}$-derivation $\nabla:\Gamma_{0}^{\infty}\lc{}A_{X}\rc\longrightarrow L^{2}\lc\Gamma\lc{}A_{X}\rc{},\int_{X}\tau d\vol\rc$ by setting

\begin{align}\label{EQ.BSP.QOT_Second_Quantisation_Parametrised_6}
\big(\nabla F\big)(T):=\big(\nabla^{h} F\big)(T)+\big(\nabla^{v} F\big)(T)
\end{align}

\noindent for all $F\in\Gamma_{0}^{\infty}\lc{}A_{X}\rc$ and $T\in X$. Equation \ref{EQ.BSP.QOT_Second_Quantisation_Parametrised_6} yields noncommutative gradient for a mixed continuous-discrete noncommutative geometry. We define continuity equations as per $2)$ in Definition \ref{DFN.Continuity_Equation} by testing on $\Gamma_{0}^{\infty}\lc{}A_{X}\rc$ and therefore admissible paths as per Definition \ref{DFN.Admissible_Paths}. Let $f$ be symmetric representing function of an operator mean and $\theta\in [0,1]$. For all $F,G\in L^{1}\lc\Gamma\lc{}A_{X}\rc{},\int_{X}\tau d\vol\rc$, we define closed positive unbounded quadratic form on $L^{2}\lc\Gamma\lc{}A_{X}\rc{},\int_{X}\tau d\vol\rc$ as per Theorem \ref{THM.NCD_Operator_Compressed_PMO} by setting

\begin{align}\label{EQ.BSP.QOT_Second_Quantisation_Parametrised_7}
Q_{F^{\flat},G^{\flat}}^{f,\theta}\lc{}W\rc{}:=\int_{X}\mathcal{I}_{\AII\lc{}H\lb{}J_{T}\rb\rc{},\AII\lc{}H\lb{}J_{T}\rb\rc{}}^{f,\theta}\big(F(U)^{\flat},G(U)^{\flat},W(U)^{\flat}\big)d\vol
\end{align}

\noindent for all $W\in L^{2}\lc\Gamma\lc{}A_{X}\rc{},\int_{X}\tau d\vol\rc$. If we show Equation \ref{EQ.BSP.QOT_Second_Quantisation_Parametrised_7} extends to a quasi-entropy for $A_{X}$, then it defines energy functionals as per Definition \ref{DFN.Energy_Functional}. Altogether, we define dynamic transport distances as per Definition \ref{DFN.QOT_Distance}.\par
We call these generalised quantum optimal transport distances parametrised by gauge fields, or parametrised quantum optimal transport distances. If $\nabla^{h}=0$, then $\nabla=\nabla^{v}$ determines a mean quantum optimal transport for normalised averages of positive bounded functionals on CAR-algebras as per Example \ref{BSP.QOT_Second_Quantisation}. The latter is recovered as the singular case of dimension zero given by $X=\lset\textrm{pt}\rset$ and $d\vol=\delta_{\textrm{pt}}$. We therefore know it is indeed $\nabla^{v}$ allowing for cross-fibre transport. How much non-ergodicity in the AF-$C^{*}$-setting is in fact due to a lack of such cross-fibre transport is unknown to us.\par
We generalise Equation \ref{EQ.SSEC.QOT_DT_BSP_17} and define the internalised spectral action. For all $F\in\SII\lc\Gamma\lc{}A_{X}\rc\rc$, $\Ent^{\int_{X}\tau d\vol}(F):=\Ent\lc{}F,\int_{X}\tau d\vol\rc\in [-\infty,\infty]$ is the relative entropy of $F$ w.r.t.~$\int_{X}\tau d\vol$ as per Equation \ref{EQ.PRP.Rel_Ent_Cstar_II_1}. Assume smooth map $\Inn:X\longrightarrow\SII\lc\AII(H)\rc$ using the $w^{*}$-topology on $\SII\lc\AII(H)\rc$. For all $T\in X$, we rewrite Equation \ref{EQ.SSEC.QOT_DT_BSP_17} as

\begin{align}\label{EQ.BSP.QOT_Second_Quantisation_Parametrised_8}
S_{b}^{\Inn}\lc\textrm{id}_{\AII\lc{}H\lb{}J_{T}\rb\rc{}}\delta_{T}\rc{}=-\Ent\hspace{-0.0375cm} \bigg(\big(\rho_{J_{T}}^{-1}\big)^{*}\big(\Inn(T)\big)\delta_{T},\int_{X}\tau d\vol\bigg).
\end{align}

\noindent Note the right-hand side of Equation \ref{EQ.BSP.QOT_Second_Quantisation_Parametrised_8} is infinite in general since it evaluates Dirac measures. We consider a more direct definition by further subsuming precomposition in Equation \ref{EQ.BSP.QOT_Second_Quantisation_Parametrised_8} using more general internalisation maps. If we have weakly smooth map $\Inn:\SII\lc\Gamma\lc{}A_{X}\rc\rc\longrightarrow\SII\lc\Gamma\lc{}A_{X}\rc\rc$ w.r.t.~the $w^{*}$-topology on $\SII\lc\Gamma\lc{}A_{X}\rc\rc$, then we define its associated internalised spectral action by setting

\begin{align}\label{EQ.BSP.QOT_Second_Quantisation_Parametrised_9}
S_{b}^{\Inn}(\mu):=-\Ent\hspace{-0.0375cm} \bigg(\Inn(\mu),\int_{X}\tau d\vol\bigg)
\end{align}

\noindent for all $\mu\in\SII\lc\Gamma\lc{}A_{X}\rc\rc$. Note Equation \ref{EQ.BSP.QOT_Second_Quantisation_Parametrised_9} transforms the spectral action into an action functional on generalised gauge fields rather than mere points. An obvious but trivial choice for the internalisation map $\Inn:\SII\lc\Gamma\lc{}A_{X}\rc\rc\longrightarrow\SII\lc\Gamma\lc{}A_{X}\rc\rc$ is the identity map.\par
We see our choice of internalisation map is essential. Specific forms, e.g.~all of those utilising $\beta:X\longrightarrow\mathbb{R}$ due to its use for density operators as per Equation \ref{EQ.BSP.QOT_Second_Quantisation_3}, are of interest. If we have a regularisation property for internalisation maps w.r.t.~a weak Riemannian geometry in the logarithmic mean setting, then Equation \ref{EQ.BSP.QOT_Second_Quantisation_Parametrised_11} suggests a gradient flow description of the dynamics of generalised gauge fields driven by the internalised spectral action. For details on relative entropy for $W^{*}$-algebras, the logarithmic mean setting and non-spatial lower Ricci bounds, we refer to Chapter \ref{CH.L2W}.\par
Let $\mathcal{S}_{-1}^{\NI,\infty}\lc\Gamma\lc{}A_{X}\rc\rc$ as per $2)$ in Definition \ref{DFN.Cstar_Trace_Abstract_State_Space} equipped with $\|.\|_{\infty}$-topology. We may have to weaken it. Assume it has a weak Riemannian metric induced by the given quasi-entropy as per Equation \ref{EQ.BSP.QOT_Second_Quantisation_Parametrised_7} in the logarithmic mean setting analogous to the finite-dimensional case as per Definition \ref{DFN.RM_II}. We use identical notation. Assume $\Delta:=\nabla^{*}\nabla$ has $\ker\Delta=\langle 1_{A_{X}}\rangle_{\mathbb{C}}$ to avoid non-ergodicity. Moreover, we demand stronger smooth regularisation $\Inn:\SII\lc\Gamma\lc{}A_{X}\rc\rc\longrightarrow\mathcal{S}_{-1}^{\NI,\infty}\lc\Gamma\lc{}A_{X}\rc\rc$ from $\|.\|_{\Gamma\lc{}A_{X}\rc^{*}}$-~to $\|.\|_{\infty}$-topology. We see a weaker topology weakens our regularity assumption. Let $F:\lc{}-\varepsilon,\varepsilon\rc\longrightarrow\SII\lc\Gamma\lc{}A_{X}\rc\rc$ with $F(0)=\mu$ be smooth. Equation \ref{EQ.BSP.QOT_Second_Quantisation_Parametrised_9} implies

\begin{align}\label{EQ.BSP.QOT_Second_Quantisation_Parametrised_10}
\restr{0.925}{\frac{d}{d\varepsilon}}{\varepsilon=0}S_{b}^{\Inn}\lc{}F(\varepsilon)\rc{}=-d_{\Inn(\mu)}\textrm{Ent}^{\int_{X}\tau d\vol}\bigg(d_{\mu}\Inn\big(\dot{F}(0)\big)\bigg).
\end{align}

\noindent We want $\grad_{\eta}\Ent^{\int_{X}\tau d\vol}=\lc\sharp\Delta\eta\rc^{\flat}$ for all $\eta\in\mathcal{S}_{-1}^{\NI,\infty}\lc\Gamma\lc{}A_{X}\rc\rc\cap\lc\dom\Delta\rc^{\flat}$ in direct analogy to Equation \ref{EQ.PRP.L2W_Log_Mean_QNE_GradFlow_2} in the proof of Proposition \ref{PRP.L2W_Log_Mean_QNE_GradFlow}. If we do lift said finite-dimensional pointwise case, then, for $\xi:=\lc\int\tau\lc{}1_{A_{X}}\rc{}d\vol\rc^{-1}1_{A_{X}}\in\mathcal{S}^{\NI}\lc\Gamma\lc{}A_{X}\rc\rc$, Equation \ref{EQ.BSP.QOT_Second_Quantisation_Parametrised_10} is

\begin{align}\label{EQ.BSP.QOT_Second_Quantisation_Parametrised_11}
\restr{0.925}{\frac{d}{d\varepsilon}}{\varepsilon=0}S_{b}^{\Inn}\lc{}F(\varepsilon)\rc{}=-g_{\Inn(\mu)}^{\xi}\bigg(d_{\mu}\Inn\big(\dot{F}(0)\big),\lc\Delta\sharp\Inn(\mu)\rc^{\flat}\bigg).
\end{align}

\noindent If regularisation allows pointwise adjoining of the derivatives in Equation \ref{EQ.BSP.QOT_Second_Quantisation_Parametrised_11} s.t.~we adjoin to the given quasi-entropy precomposed with a well-behaved map, then we may use it to express metric slopes as per Equation \ref{EQ.SSEC.L2W_EVI_Equivalence_1} and control any $\EVI_{\lambda}$-gradient flow of $S_{b}^{\Inn}$ \cite{BK.Amb_Gig_Sav.2008.Classical_OT_GradFlow}\cite{ART.Mur_Sav.2020.Classical_OT_EVI}. If we show lower Ricci bounds are Hessian lower bounds as per $\HI\rc$ in Definition \ref{DFN.L2W_EVI_Equivalence} for our choice of weak Riemannian geometry, then a given one may use the above adjoining relation to impact the dynamics of $\SII\lc\Gamma\lc{}A_{X}\rc\rc$ driven by $S_{b}^{\Inn}$.\par
We must solve key technical challenges, ranging from our initial construction to choice of internalisation map, its interplay with a suitable weak Riemannian structure and the $\EVI_{\lambda}$-gradient flow property for $\Ent^{\int_{X}\tau d\vol}$. We may therefore seek to relax the problem as follows. We use, as in the AF-$C^{*}$-setting, canonical $C^{*}$-bimodule structures. If we instead consider general u.s.c.~$C^{*}$-bundles and $C^{*}$-bimodule actions s.t.~each fibre in Equation \ref{EQ.BSP.QOT_Second_Quantisation_Parametrised_1} is a tracial AF-$C^{*}$-algebra, then we also consider Equation \ref{EQ.BSP.QOT_Second_Quantisation_Parametrised_6} for more general noncommutative gradients. Such disintegration of tracial $W^{*}$-algebras into direct integrals of factors follows from the von Neumann disintegration theorem in operator theory \lc{}cf.~Theorem IV.8.21 in \cite{BK.Tak.1979.OpAlg_I}\rc{}. We see fundamental example classes using tracial AF-$C^{*}$-algebras generating hyperfinite factors of type I and II by $\sigma$-weak closure are of particular interest. We thereby define general parametrised quantum optimal transport. We view quantum optimal transport as its pointwise case since the latter is the singular case of dimension zero.
\end{bsp}


\section{Accessibility components}\label{SEC.QOT_AC}

Accessibility components of quantum optimal transport distances are complete geodesic length-metric spaces. Metric geometry reduces to accessibility components. There may exist uncountable infinitely many since sets of states at finite distance have identical fixed parts under noncommutative heat semigroups of quantum Laplacians. Assuming spectral gaps of quantum Laplacians and fixed parts, we use such fixed parts to classify accessibility components of square integrable normal states. We in turn use the latter classification for the coarse graining process since its assumptions are satisfied for all accessibility components in the finite-dimensional setting.\par
Classification uses regularisation of normal states under heat flow. Assuming fixed parts with integrable support, we show heat flow instantaneously regularises normal states to be, possibly unboundedly, invertible up to fixed part. This uses compatibility with compression and finite-dimensional approximation. Note we avoid any regularity assumptions for noncommutative heat semigroups. Under assumptions as above, we use such regularisation for classification since spectral gaps of square integrable normal fixed parts imply integrable support. In the logarithmic mean setting and assuming finitely supported fixed parts, we further use it to show heat flow induces finite-energy admissible paths for all states with finite quantum relative entropy.\par
We show classification and regularisation by passing through the finite-dimensional setting. In the latter setting, accessibility components are norm closed convex subsets of states having identical fixed part. States at finite distance have support projections in the unique compressed $C^{*}$-subalgebra which is given by compressing with the support projection of their common fixed part. Relative interiors consist of all invertible states on, resp.~densities in, such a given compressed $C^{*}$-subalgebra. They are also connected Riemannian manifolds with Riemannian metric induced by the given quasi-entropy. Using finite-dimensionality, we directly verify heat flow yields finite energy paths from relative boundaries to relative interiors. We thereby connect all states with identical fixed part. This yields classification and regularisation in the finite-dimensional case. Under assumptions as above, we extend regularisation and classification to the general case. We require the notion of reducible support as finite-dimensional approximation of support projections. We show it is implied by integrable support. We are therefore able to pass through the finite-dimensional setting.

\medskip

\noindent\textbf{Structure.} In Subsection \ref{SSEC.QOT_AC_SUPP}, we review support projections of normal states, as well as spectral gaps. We introduce the notion of reducible support. In Subsection \ref{SSEC.QOT_AC_HSG}, we discuss both completely Markovian semigroups and Lindblad master equations, our use of quantum Fokker-Planck equations, and subsequently study noncommutative heat semigroups of quantum Laplacians. In Subsection \ref{SSEC.QOT_AC_RM}, we apply the latter to classify accessibility components of square integrable normal states as discussed above.


\subsection{Support projections of normal states}\label{SSEC.QOT_AC_SUPP}

We review canonical order-preserving bijections from projections of $W^{*}$-algebras to faces of normal state spaces. They are determined by support projections of normal states.\par
In the AF-$C^{*}$-setting, reducible support is finite-dimensional approximation of such support projections. Theorem \ref{THM.AF_Support_Projection} shows integrable support implies reducible support as required. Standard references for convex geometry of norm closed convex subsets in pre-duals of $W^{*}$-algebras are \cite{BK.Alf_Shu.2001.OpAlg_State_Spaces_I}\cite{BK.Alf_Shu.2003.OpAlg_State_Spaces_II}. Standard reference for differential and Riemannian geometry is \cite{BK.Lan.1995.Riemannian_Manifolds}.


\subsubsection*{Faces of normal state spaces}

Lemma \ref{LEM.Support_Projection} represents faces of normal state spaces of abstract tracial $C^{*}$-algebras as per Definition \ref{DFN.Cstar_Trace_Abstract_State_Space} and Remark \ref{REM.Cstar_Trace_Abstract_State_Space}. This uses normal state spaces of compressed $C^{*}$-subalgebras and their canonical inclusions as per $1)$ in Proposition \ref{PRP.Cstar_Trace_Abstract_Dualisation_II}. We use the modified standard pairing, in particular their flat and sharp operators as per Definition \ref{DFN.Wstar_Trace_MSP_Musical} and Remark \ref{REM.Wstar_Trace_MSP}.\par
Let $(M,\tau)$ be a tracial $W^{*}$-algebra and $A\subset M$ a $\sigma$-weakly dense $C^{*}$-subalgebra. Ergo $M=L^{\infty}(A,\tau)$ and $M_{*}=L^{1}(A,\tau)$. For all $x\in L^{1}(A,\tau)_{+}$, we have unique carrier projection $\supp x\in L^{\infty}(A,\tau)$ of $\lset{}x^{\flat}\rset$ \lc{}cf.~Definition 3.20 and Lemma 3.21 in \cite{BK.Alf_Shu.2001.OpAlg_State_Spaces_I}\rc{}. Each $\supp x$ is the minimal projection in $L^{\infty}(A,\tau)$ s.t.~$x=x\cdot \supp x$ holds. If we have $x=xp$ for a projection $p\in L^{\infty}(A,\tau)$, then $\supp x\leq p$. Note $x=xp$, $x=px$ and $x=pxp$ are equivalent.

\begin{dfn}\label{DFN.Support_Projection_I}
Let $x\in L^{1}(A,\tau)_{+}$. 

\begin{itemize}
\item[1)] The carrier projection $\supp x\in L^{\infty}(A,\tau)$ of $\lset{}x^{\flat}\rset$ is the support projection of $x$.

\item[2)] If $x\in L^{0}(N,\tau)$ for $N\subset \lc{}L^{\infty}(A,\tau),\tau\rc$, then we say that $\suppcN x:=1_{N}-\supp x$ is the kernel projection of $x$ in $N$.
\end{itemize}
\end{dfn}

\begin{ntn}\label{NTN.Support_Projection}
We suppress $N$ in Definition \ref{DFN.Support_Projection_I} if $N=L^{\infty}(A,\tau)$.
\end{ntn}

\begin{prp}\label{PRP.Support_Projection_I}
Let $N\subset \lc{}L^{\infty}(A,\tau),\tau\rc$.

\begin{itemize}
\item[1)] Let $x\in L^{1}(N,\tau)_{+}$. We have $\supp x\in N$. Furthermore, we have $\suppc x\in N[1_{A}]$ and $\suppcN x=\comunit\suppc x\in N$.

\item[2)] Let $x,y\in L^{1}(A,\tau)_{+}$. If $\tau\lc{}yp\rc{}=0$ for all projections $p\in L^{\infty}(A,\tau)$ s.t.~$\tau\lc{}xp\rc{}=0$, then $\supp y\leq\supp x$.
\end{itemize}
\end{prp}
\begin{proof}
Since $\comunit 1_{A}=1_{N}$, we know $1)$ by definition. Get $2)$ by Lemma 3.25 in \cite{BK.Alf_Shu.2001.OpAlg_State_Spaces_I}.
\end{proof}

Proposition \ref{PRP.Support_Projection_II} shows support projections are invariant under change of canonical left-~and right-actions. Let $N\subset \lc{}L^{\infty}(A,\tau),\tau\rc$ and $x\in L^{0}(N,\tau)_{+}$. Using the latter and following Remark \ref{REM.FC_Injectivity}, note $2)$ in Lemma \ref{LEM.Wstar_CLRA_FC} shows

\begin{align}\label{EQ.SSEC.QOT_AC_SUPP_1}
\Gamma_{x,N}\lc\chi_{(0,\infty]}\rc{}=L_{N}^{-1}\lc\pi_{\im L_{x,N}}^{A}\rc{}=R_{N}^{-1}\lc\pi_{\im R_{x,N}}^{A}\rc{}
\end{align}

\noindent and

\begin{align}\label{EQ.SSEC.QOT_AC_SUPP_2}
\Gamma_{x,N}(\delta_{0})=L_{N}^{-1}\lc\pi_{\ker L_{x,N}}^{A}\rc{}=R_{N}^{-1}\lc\pi_{\ker R_{x,N}}^{A}\rc{}.
\end{align}

\begin{prp}\label{PRP.Support_Projection_II}
Let $N\subset \lc{}L^{\infty}(A,\tau),\tau\rc$. For all $x\in L^{1}(N,\tau)_{+}$, we have

\begin{itemize}
\item[1)] $\supp x=\Gamma_{x,N}\lc\chi_{(0,\infty]}\rc{}=L_{N}^{-1}\lc\pi_{\im L_{x,N}}^{A}\rc{}=R_{N}^{-1}\lc\pi_{\im R_{x,N}}^{A}\rc$, \phantom{\bigg)}

\item[2)] $\suppcN x=\Gamma_{x,N}(\delta_{0})=L_{N}^{-1}\lc\pi_{\ker L_{x,N}}^{A}\rc{}=R_{N}^{-1}\lc\pi_{\ker R_{x,N}}^{A}\rc$, \phantom{\bigg)}

\item[3)] $L_{\supp x,N}=\com\hspace{-0.055cm}_{L^{2}(N,\tau)}L_{\supp x,L^{\infty}(A,\tau)}$ and $R_{\supp x,N}=\com\hspace{-0.055cm}_{L^{2}(N,\tau)}R_{\supp x,L^{\infty}(A,\tau)}$. \phantom{\bigg)}
\end{itemize}
\end{prp}
\begin{proof}
Let $x\in L^{1}(N,\tau)_{+}$. Note we have $x=x\Gamma_{x,N}\lc\chi_{(0,\infty]}\rc$ and $0=x\Gamma_{x,N}(\delta_{0})$ by functional calculus. Thus minimality of support projections implies $\supp x\leq\Gamma_{x,N}\lc\chi_{(0,\infty]}\rc$, hence $\supp x\cdot \Gamma_{x,N}\lc\chi_{(0,\infty]}\rc{}=\supp x$ since both are projections. We prove the converse. For all $u\in L^{2}(A,\tau)$, let $\{u_{n}\}_{n\in\mathbb{N}}\subset\dom L_{x,N}$ s.t.~$\pi_{\im L_{x}}^{A}(u)=\|.\|_{\tau}$-$\lim_{n\in\mathbb{N}}xu_{n}$. Using the latter in each case and further $\supp x\cdot \Gamma_{x,N}\lc\chi_{(0,\infty]}\rc{}=\supp x$, Equation \ref{EQ.SSEC.QOT_AC_SUPP_1} lets us calculate

\begin{align*}
\dblv{}\supp x\cdot u\dblv_{\tau} & = \dblv{}L_{\supp x\cdot\Gamma_{x,N}\lc\chi_{(0,\infty]}\rc{},N}\lc\pi_{L^{2}(N,\tau)}^{A}(u)\rc\dblv_{\tau} \phantom{\Bigg)} \\
& = \dblv{}L_{\supp x,N}\lc\pi_{\im L_{x,N}}^{A}(u)\rc\dblv_{\tau} \phantom{\Bigg)} \\
& = \lim_{n\in\mathbb{N}}\hspace{0.025cm} \dblv{}L_{\supp x,N}\lc{}xu_{n}\rc\dblv_{\tau} \phantom{\Bigg)} \\
& = \dblv{}\pi_{\im L_{x}}^{A}(u)\dblv_{\tau}=\dblv{}\Gamma_{x,N}\lc\chi_{(0,\infty]}\rc{}\cdot u\dblv_{\tau} \phantom{\Bigg)}
\end{align*}

\noindent for all $u\in L^{2}(A,\tau)$. Since $\supp x\leq\Gamma_{x,N}\lc\chi_{(0,\infty]}\rc$, get $\supp x=\Gamma_{x,N}\lc\chi_{(0,\infty]}\rc$ and therefore $1)$ by Equation \ref{EQ.SSEC.QOT_AC_SUPP_1}. Then $\suppcN x=1_{N}-\Gamma_{x,N}\lc\chi_{(0,\infty]}\rc{}=\Gamma_{x,N}(\delta_{0})$ by functional calculus and we have $2)$ by Equation \ref{EQ.SSEC.QOT_AC_SUPP_2}. Using $1)$, get $3)$ by Corollary \ref{COR.Wstar_Compression_Preservation_I}.
\end{proof}

Theorem 3.35 in \cite{BK.Alf_Shu.2001.OpAlg_State_Spaces_I} classifies norm closed convex subsets of normal state space using support projections. We review this below for abstract tracial $C^{*}$-algebras. Let $V$ be a normed vector space and $K\subset V$ a norm closed convex subset. Its relative interior
 
\begin{align}\label{EQ.SSEC.QOT_AC_SUPP_3}
\relint K=\lset\mu\in K\ \vset\ \forall\eta\in K\ \exists t>1:\ t\mu+\lc{}1-t\rc\eta\in K\rset{}
\end{align}

\noindent is open, and its relative boundary $\partial K=K\setminus\relint K$ is closed in relative topology. A norm closed convex subset $\mathcal{F}\subset K$ is a face of $K$ if for all $x,y\in K$, we know $\lc{}1-t\rc{}x+ty\in\mathcal{F}$ for any $t\in (0,1)$ implies $x,y\in\mathcal{F}$.

\begin{lem}\label{LEM.Support_Projection}
For all projections $p\in L^{\infty}(A,\tau)$, we know

\begin{align}
\mathcal{F}_{A}\lc{}p\rc{}:=\lset{}x\in L^{1}(A,\tau)_{+}\ \vset\ \|x\|_{1}=1,\ \supp x\leq p\rset^{\flat}=\mathcal{S}^{\NI}(A[p])
\end{align}

\noindent is a face of $\mathcal{S}^{\NI}(A)$. Furthermore, the map $p\mapsto\mathcal{F}_{A}\lc{}p\rc$ from projections in $L^{\infty}(A,\tau)$ to faces of $\mathcal{S}^{\NI}(A)$ is an order-preserving bijection.
\end{lem}
\begin{proof}
We use $1)$ in Proposition \ref{PRP.Cstar_Trace_Abstract_Dualisation_II} here and throughout our discussion. Using $1)$ in Theorem 3.35 and Lemma 3.21 in \cite{BK.Alf_Shu.2001.OpAlg_State_Spaces_I}, we obtain a face

\begin{align}\label{EQ.LEM.Support_Projection_1}
\mathcal{F}_{A}\lc{}p\rc{}=\lset\mu\in\mathcal{S}^{\NI}(A)\ \vset\ \mu\lc{}p\rc{}=1\rset{}=\lset{}x\in L^{1}(A,\tau)_{+}\ \vset\ \|x\|_{1}=1,\ \supp x\leq p\rset^{\flat}
\end{align}

\noindent of $\mathcal{S}^{\NI}(A)$ in each case. Theorem 3.35 in \cite{BK.Alf_Shu.2001.OpAlg_State_Spaces_I} states the map $p\mapsto\mathcal{F}_{A}\lc{}p\rc$ from projections in $L^{\infty}(A,\tau)$ to faces of $\mathcal{S}^{\NI}(A)$ is an order-preserving bijection. Let $p\in L^{\infty}(A,\tau)$ be a projection. If $x\in\mathcal{F}_{A}\lc{}p\rc$, then $\supp x\leq p$ implies $\supp x\cdot p=\supp x$ and therefore $x=xp$ by minimality. Thus $x\in L^{0}(A[p],\tau)_{+}$ by Lemma \ref{LEM.Cstar_Trace_Abstract_Projection}, hence $x^{\flat}\in\mathcal{S}^{\NI}(A[p])$. The converse follows because $x^{\flat}\in\mathcal{S}^{\NI}(A[p])$ likewise implies $x=xp$, which in turn implies $\supp x\leq p$ by Lemma 3.21 in \cite{BK.Alf_Shu.2001.OpAlg_State_Spaces_I}.
\end{proof}

\begin{dfn}\label{DFN.Support_Projection_II}
For all $\mu\in L^{1}(A,\tau)_{+}^{\flat}$, set

\begin{itemize}
\item[1)] $\supp\mu:=\supp\sharp\mu$ and call $\supp\mu$ the support projection of $\mu$,

\item[2)] $\mathcal{F}_{A}(\mu):=\mathcal{F}_{A}\lc\supp\mu\rc$ and call $\mathcal{F}_{A}(\mu)$ the face of $\mu$ on $A$.
\end{itemize}
\end{dfn}

\begin{rem}\label{REM.Support_Projection_I}
Let $\mu\in L^{1}(A,\tau)^{\flat}$. For all projections $p\in L^{\infty}(A,\tau)$ s.t.~$\supp\mu\leq p$, we have $\mu\in L^{1}(A[p],\tau)^{\flat}$ and therefore $\FII_{A}(\mu)=\FII_{A[p]}(\mu)$ by Lemma \ref{LEM.Support_Projection}.
\end{rem}

\begin{cor}\label{COR.Support_Projection_I}
Let $p,q\in L^{\infty}(A,\tau)$ be projections. 

\begin{itemize}
\item[1)] We have $p\leq q$ if and only if $\mathcal{S}^{\NI}(A[p])\subset\mathcal{S}^{\NI}\lc{}A\lb{}q\rb\rc$.

\item[2)] Assume $p\leq q$. If $K\subset\mathcal{S}^{\NI}(A[p])$ is a face, then $K\subset\mathcal{S}^{\NI}\lc{}A\lb{}q\rb\rc$ is a face.
\end{itemize}
\end{cor}
\begin{proof}
Apply Lemma \ref{LEM.Support_Projection}.
\end{proof}

\begin{rem}\label{REM.Support_Projection_II}
Let $p\in L^{\infty}(A,\tau)$ be a projection. If $\mu\in\mathcal{S}^{\NI}(A[p])$, then $\supp\mu\leq p$ by Lemma \ref{LEM.Support_Projection} and therefore $\FII_{A}(\mu)\subset\mathcal{S}^{\NI}(A[p])$ by $1)$ in Corollary \ref{COR.Support_Projection_I}. If $K\subset\mathcal{S}^{\NI}(A[p])$ is a face, then $K\subset\mathcal{S}^{\NI}(A)$ is a face by $2)$ in Corollary \ref{COR.Support_Projection_I}.
\end{rem}

\begin{cor}\label{COR.Support_Projection_II}
For all $\mu\in\mathcal{S}^{\NI}(A)$, we have $\FII_{A}(\mu)=\lset\mu\rset$ if and only if $\mu$ is pure.
\end{cor}
\begin{proof}
Let $\mu\in\mathcal{S}^{\NI}(A)$. If $\FII_{A}(\mu)=\lset\mu\rset$, then purity of $\mu$ follows by the face property. Assume $\mu$ is pure. We know $K:=\lset\mu\rset\subset\mathcal{S}^{\NI}(A[p])$ is a face. Using $2)$ in Corollary \ref{COR.Support_Projection_I} and following Remark \ref{REM.Support_Projection_II}, Lemma \ref{LEM.Support_Projection} yields unique projection $q\in L^{\infty}(A[p],\tau)$ s.t.~we have $K=\SII\lc{}A\lb{}q\rb\rc$. Since $\mu\in K$, the lemma further shows $\supp\mu\leq q$ and therefore $\FII_{A}(\mu)\subset K$. We obtain $\FII_{A}(\mu)=\lset\mu\rset$ as claimed.
\end{proof}

Let $p\in L^{\infty}(A,\tau)$ be a projection. Note $A[p]_{+}^{*}\cap\GL\lc{}L^{\infty}(A[p],\tau)\rc^{\flat}\subset A[p]_{h}^{*}$ open in norm topology. Using real vector space structure, we see $A[p]_{+}^{*}\cap\GL\lc{}L^{\infty}(A[p],\tau)\rc^{\flat}$ is a Banach manifold. We have $\mathcal{S}_{-1}^{\NI,\infty}(A[p])=\mathcal{S}^{\NI,\infty}(A[p])\cap\GL\lc{}L^{\infty}(A[p],\tau)\rc^{\flat}$ by boundedness.\par


\pagebreak


\begin{cor}\label{COR.Support_Projection_III}
Let $p\in L^{\infty}(A,\tau)$ be a projection s.t.~$\tau\lc{}p\rc{}<\infty$.

\begin{itemize}
\item[1)] We have embedded Banach submanifold

\begin{align}\label{EQ.COR.Support_Projection_III_1}
\mathcal{S}_{-1}^{\NI,\infty}(A[p])=\relint\mathcal{S}^{\NI,\infty}(A[p])\subset A[p]_{+}^{*}\cap\GL\lc{}L^{\infty}(A[p],\tau)\rc^{\flat}.    
\end{align}

\begin{reapply}
\end{reapply}

\item[2)] For all $\mu\in\mathcal{S}^{\NI}(A[p])$, we have

\begin{itemize}
\item[2.1)] $\mathcal{F}_{A}(\mu)=\mathcal{S}^{\NI}(A[p])$ if and only if $\mu\in\mathcal{S}_{>0}^{\NI}(A[p])$,

\item[2.2)] $\mathcal{F}_{A}(\mu)\subset\partial\mathcal{S}^{\NI}(A[p])$ if and only if $\mu\notin\mathcal{S}_{>0}^{\NI}(A[p])$.
\end{itemize}

\begin{reapply}
\end{reapply}

\end{itemize}
\end{cor}
\begin{proof}
Set $\xi_{p}:=\tau\lc{}p\rc^{-1}p^{\flat}$. We show $1)$. For all $\mu\in\mathcal{S}^{\NI,\infty}(A[p])$, we have $\mu\in\mathcal{S}_{-1}^{\NI,\infty}(A[p])$ if and only if $\sharp\mu\in\GL\lc{}L^{\infty}(A[p],\tau)\rc$ by $2)$ in Definition \ref{DFN.Cstar_Trace_Abstract_State_Space}. In particular, the equivalence ensures $\xi_{p}\in\mathcal{S}_{-1}^{\NI,\infty}(A[p])$. For all $\mu\in\relint\mathcal{S}^{\NI,\infty}(A[p])$, there exists $t>1$ s.t.~

\begin{align}\label{EQ.COR.Support_Projection_III_2}
\mu\geq\frac{t-1}{t}\xi_{p}\geq 0.
\end{align}

\noindent Since $\xi_{p}\in\mathcal{S}_{-1}^{\NI,\infty}(A[p])$, note Equation \ref{EQ.COR.Support_Projection_III_2} shows $\relint\mathcal{S}^{\NI,\infty}(A[p])\subset\mathcal{S}_{-1}^{\NI,\infty}(A[p])$. We directly verify the converse. Thus Equation \ref{EQ.COR.Support_Projection_III_1} holds, hence

\begin{align}\label{EQ.COR.Support_Projection_III_3}
\relint\mathcal{S}^{\NI,\infty}(A[p])=\lc\restr{0.925}{\tau}{A[p]_{+}^{*}\cap\GL\lc{}L^{\infty}(A[p],\tau)\rc^{\flat}}\rc^{-1}(1).
\end{align}

\noindent Equation \ref{EQ.COR.Support_Projection_III_3} implies $1)$ by the submersion theorem \cite{BK.Lan.1995.Riemannian_Manifolds}.\par
We show $2)$. For all $\mu\in\mathcal{S}^{\NI}(A[p])$, $\mu\in\mathcal{S}_{>0}^{\NI}(A[p])$ if and only if $\Gamma_{\sharp\mu,L^{\infty}(A[p],\tau)}(\delta_{0})=0$ by $1)$ in Definition \ref{DFN.Cstar_Trace_Abstract_State_Space}. Proposition \ref{PRP.Support_Projection_II} further shows the latter is equivalent to $\supp\mu=p$. Then $1)$ in Corollary \ref{COR.Support_Projection_I} yields $2.1)$. Lemma \ref{LEM.Support_Projection} shows $\supp\mu\leq p$ in each case, i.e.~$\mathcal{F}_{A}(\mu)\subset\mathcal{S}^{\NI}(A[p])$ by $1)$ in Corollary \ref{COR.Support_Projection_I}. Using the latter and $2.1)$, note Equation \ref{EQ.COR.Support_Projection_III_2} for $\relint\mathcal{S}^{\NI}(A[p])$ derives $2.2)$ by contradiction.
\end{proof}


\subsubsection*{Support projections in the AF-$\mathbf{C}^{*}$-setting}

Definition \ref{DFN.AF_Support_Projection_II} gives reducible support. Theorem \ref{THM.AF_Support_Projection} shows integrable support implies reducible support. Spectral gaps imply integrable support. Lemma \ref{LEM.AF_Support_Projection_II} shows spectral gaps of square integrable positive elements are limits of spectral gaps of their restrictions. This shows the utility of assuming spectral gaps in order to use finite-dimensional approximation.\par
Let $H$ be a Hilbert space. Let $(A,\tau)$ be a tracial AF-$C^{*}$-algebra.

\begin{lem}\label{LEM.AF_Support_Projection_I}
Let $T=\sr$-$\lim_{n\in\mathbb{N}}T_{n}$ on $H$. If $T_{n}\geq 0$ for all $n\in\mathbb{N}$, then

\begin{align}\label{EQ.LEM.AF_Support_Projection_I_1}
0\leq\limsup_{j\in\mathbb{N}}\hspace{0.025cm} \dblv{}\lc\chi_{(0,\infty]}(T)-\chi_{(0,\infty]}(T_{n})\rc{}(u)\dblv_{H}\leq 2\dblv{}\delta_{0}(T)(u)\dblv_{H}
\end{align}

\noindent for all $u\in H$.
\end{lem}
\begin{proof}
For all $\varepsilon>0$, we define $g_{\varepsilon}\in C_{b}\lc{}[0,\infty)\rc$ by setting

\begin{align*}
g_{\varepsilon}(t):=
\begin{cases}
\varepsilon^{-1}t & \If\ t\leq\varepsilon, \\
1 & \Else.
\end{cases}
\end{align*}

\noindent By construction, get $g_{\varepsilon}(0)=0$, $\dblv{}g_{\varepsilon}\dblv_{\infty}=1$ and $0\leq\chi_{(0,\infty]}-g_{\varepsilon}\leq I-g_{\varepsilon}$ in each case. We moreover have pointwise convergence $\chi_{(0,\infty]}=\lim_{\varepsilon\downarrow 0}g_{\varepsilon}$ on $[0,\infty)$. Let $S\in\UBII(H)_{+}$. Note $\chi_{(0,\infty]}(S)=\pi_{\im S}$ and $\delta_{0}(S)=\pi_{\ker S}$ \lc{}cf.~Remark \ref{REM.FC_Injectivity}\rc{}. Thus $\chi_{(0,\infty]}(S)=\s$-$\lim_{\varepsilon\downarrow 0}g_{\varepsilon}(S)$ by uniform boundedness, hence $\delta_{0}(S)=\s$-$\lim_{\varepsilon\downarrow 0}1-g_{\varepsilon}(S)$.\par
Let $u\in H$ and $\varepsilon>0$. Then $\dblv{}\lc\chi_{(0,\infty]}(T)-g_{\varepsilon}(T)\rc{}(u)\dblv_{H}\leq \dblv{}\lc{}I-g_{\varepsilon}(T)\rc{}(u)\dblv_{H}$, as well as $\dblv{}\lc\chi_{(0,\infty]}(T_{n})-g_{\varepsilon}(T)\rc{}(u)\dblv_{H}\leq \dblv{}\lc{}I-g_{\varepsilon}(T_{n})\rc{}(u)\dblv_{H}$ for all $n\in\mathbb{N}$, by functional calculus. For all $n\in\mathbb{N}$, we therefore bound $\dblv{}\lc\chi_{(0,\infty]}(T)-\chi_{(0,\infty]}(T_{n})\rc{}(u)\dblv_{H}$ from above by

\begin{align}\label{EQ.LEM.AF_Support_Projection_I_2}
\dblv{}\big(I-g_{\varepsilon}(T)\big)(u)\dblv_{H}+\dblv{}\big(I-g_{\varepsilon}(T_{n})\big)(u)\dblv_{H}+\dblv{}\big(g_{\varepsilon}(T)-g_{\varepsilon}(T_{n})\big)(u)\dblv_{H}.
\end{align}

\noindent Since $g_{\varepsilon}\in C_{b}\lc{}[0,\infty)$, we know $g_{\varepsilon}(T)=\s$-$\lim_{n\in\mathbb{N}}g_{\varepsilon}(T_{n})$ by Lemma \ref{LEM.FC_SR}. Applying the latter to Equation \ref{EQ.LEM.AF_Support_Projection_I_2} shows

\begin{align}\label{EQ.LEM.AF_Support_Projection_I_3}
0\leq\limsup_{n\in\mathbb{N}}\hspace{0.025cm} \dblv{}\lc\chi_{(0,\infty]}(T)-\chi_{(0,\infty]}(T_{n})\rc{}(u)\dblv_{H}\leq 2\dblv{}\big(I-g_{\varepsilon}(T)\big)(u)\dblv_{H}.
\end{align}

\noindent Using $\pi_{\ker T}=\s$-$\lim_{\varepsilon\downarrow 0}I-g_{\varepsilon}(T)$, letting $\varepsilon\downarrow 0$ in Equation \ref{EQ.LEM.AF_Support_Projection_I_3} yields Equation \ref{EQ.LEM.AF_Support_Projection_I_1}.
\end{proof}

For all $T\in\UBII(H)_{+}$, we have spectral gap $\sigma(T)=\inf\hspace{0.025cm} \lset\lambda>0\ \vset\ \lambda\in\spec T\rset$ and say that $T$ has spectral gap if $\sigma(T)>0$ \lc{}cf.~Definition \ref{DFN.Spectral_Gap}\rc{}. Definition \ref{DFN.AF_Support_Projection_I} gives spectral gaps of positive measurable operators and normal positive bounded functionals. Using canonical left-~and right-actions, Proposition \ref{PRP.AF_Support_Projection} recovers spectral gaps of positive unbounded operators. Spectral gaps are invariant under compression. For details on spectral gaps of positive unbounded operators, we refer to Subsection \ref{SSEC.A_Maps_Compression}.

\begin{dfn}\label{DFN.AF_Support_Projection_I}
Let $N\subset \lc{}L^{\infty}(A,\tau),\tau\rc$.

\begin{itemize}
\item[1)] For all $x\in L^{0}(N,\tau)_{+}$, we call $\sigma(x):=\inf\hspace{0.025cm} \lset\lambda>0\ \vset\ \lambda\in\spec_{L^{\infty}(A,\tau)}x\rset$ the spectral gap of $x$. We further say that $x$ has spectral gap if $\sigma(x)>0$.

\item[2)] For all $\mu\in L^{1}(A,\tau)^{\flat}$, set $\sigma(\mu):=\sigma\lc\sharp\mu\rc$ and call $\sigma(\mu)$ the spectral gap of $\mu$. We further say that $\mu$ has spectral gap if $\sigma(\mu)>0$.
\end{itemize}
\end{dfn}

\begin{prp}\label{PRP.AF_Support_Projection}
Let $N\subset \lc{}L^{\infty}(A,\tau),\tau\rc$. For all $x\in L^{0}(N,\tau)_{+}$, we have

\begin{align}\label{EQ.PRP.AF_Support_Projection_1}
\sigma(x)=\inf\hspace{0.025cm} \lset\lambda>0\ \vset\ \lambda\in\specN x\rset{}=\sigma\lc{}L_{x,N}\rc{}=\sigma\lc{}R_{x,N}\rc{}.
\end{align}
\end{prp}
\begin{proof}
Let $x\in L^{0}(N,\tau)$. Thus $\spec_{L^{\infty}(A,\tau)}x=\specN x\cup\lset{}0\rset$ by $1)$ in Corollary \ref{COR.Wstar_Compression_Preservation_I}, hence we obtain the first identity in Equation \ref{EQ.PRP.AF_Support_Projection_1} by positivity. The second and third one follow at once from $2)$ in Proposition \ref{PRP.Wstar_CLRA_FC} and $2)$ in Lemma \ref{LEM.Wstar_CLRA_FC}.
\end{proof}


\pagebreak


\begin{rem}\label{REM.AF_Support_Projection}
Let $N\subset \lc{}L^{\infty}(A,\tau),\tau\rc$ and $\mu\in L^{1}(N,\tau)^{\flat}$. Note Proposition \ref{PRP.Wstar_NCI_VI} shows we have $\mu\in L^{1}(N,\tau)^{\flat}$ if and only if $\sharp\mu\in L^{1}(N,\tau)_{+}$. Proposition \ref{PRP.AF_Support_Projection} further implies

\begin{align}\label{EQ.REM.AF_Support_Projection_1}
\sigma(\mu)=\inf\hspace{0.025cm} \lset\lambda>0\ \vset\ \lambda\in\specN\sharp\mu\rset{}=\sigma\lc{}L_{\sharp\mu,N}\rc{}=\sigma\lc{}R_{\sharp\mu,N}\rc{}.    
\end{align}

\noindent Equation \ref{EQ.REM.AF_Support_Projection_1} holds if $N\subset \lc{}L^{\infty}(A,\tau),\tau\rc$ lies in one of the two classes of compression given in Subsection \ref{SSEC.NCDS_AF_FC}, i.e.~either if we compress to induced AF-$C^{*}$-bimodules or if we compress with projections. We use this throughout our discussion.
\end{rem}

We use the following estimate. For all $z\in L^{1}(A,\tau)_{+}$ and $j\in\mathbb{N}$, $1)$ in Proposition \ref{PRP.AF_Cstar_Bimodule_L2Inf_SR} and $2)$ in Lemma \ref{LEM.AF_Cstar_Bimodule_L1_SR} show

\begin{align}\label{EQ.SSEC.QOT_AC_SUPP_4}
0\leq L_{z_{j},A_{j}}=\pi_{j}^{A}L_{z_{j}}\pi_{j}^{A}\leq L_{\pi_{j}^{A}(\sqrt{z})}^{2}.
\end{align}

\begin{lem}\label{LEM.AF_Support_Projection_II}
For all $x\in L^{2}(A,\tau)_{+}$, we have

\begin{itemize}
\item[1)] $\sigma(x)=\lim_{j\in\mathbb{N}}\sigma(x_{j})$,

\item[2)] $\chi_{(0,\infty]}(x)=\s$-$\lim_{j\in\mathbb{N}}\chi_{(0,\infty]}(x_{j})$ if $\tau\lc\chi_{(0,\infty]}(x)\rc{}<\infty$.
\end{itemize}
\end{lem}
\begin{proof}
Following Remark \ref{REM.AF_Support_Projection}, Proposition \ref{PRP.AF_Support_Projection} ensures results for spectral gaps of positive unbounded operators likewise apply to spectral gaps of positive measurable operators, resp.~normal positive bounded functionals. Let $x\in L^{2}(A,\tau)_{+}$. For all $j\in\mathbb{N}$, get $\sigma(x_{j})=\sigma(L_{x,A_{j}})$ by Proposition \ref{PRP.AF_Support_Projection}. Theorem \ref{THM.AF_Cstar_Bimodule_CLRA_SR} states $L_{x}=\sr$-$\lim_{j\in\mathbb{N}}L_{x_{j}}$.\par
We show $1)$. Strong resolvent convergence as above implies $\limsup_{j\in\mathbb{N}}\sigma(x_{j})\leq\sigma(x)$ by Lemma \ref{LEM.Spectral_Gap_USC}. We show the converse. If $\sigma(x)=0$, then our claim follows by positivity of spectral gaps. We assume $\sigma(x)>0$ without loss of generality. Thus $x\neq 0$, hence $x_{j}\neq 0$ and therefore $\sigma(x_{j})>0$ for a.e.~$j\in\mathbb{N}$ by finite-dimensionality. We assume $x_{j}\neq 0$ and thereby $\sigma(x_{j})>0$ for all $j\in\mathbb{N}$ without loss of generality. For all $j\in\mathbb{N}$, we have $u^{j}\in A_{j}$ s.t.~$x_{j}u_{j}=\sigma(x_{j})u_{j}$ and $\|u_{j}\|_{\tau}=1$ by finite-dimensionality.\par
Let $j\in\mathbb{N}$. Let $v\in\im L_{x_{j},A_{j}}$ and $w\in A_{j}$ s.t.~$v=x_{j}w$. Note $x_{j}=\pi_{j}^{A}(x)$ by construction. Get $A_{0}\subset\dom L_{x}\cap\dom I-L_{\pi_{j}^{A}(x)}$ by square integrability. We have

\begin{align}\label{EQ.LEM.AF_Support_Projection_II_1}
v=x_{j}w=xw-\big(I-\pi_{j}^{A}\big)(x)w.
\end{align}

\noindent We know $A_{j}^{\perp}A_{j}\subset A_{j}^{\perp}$. Moreover, we have $\chi_{(0,\infty]}(x)x=x$ by functional calculus. Using each of the latter, Equation \ref{EQ.LEM.AF_Support_Projection_II_1} implies

\begin{align}\label{EQ.LEM.AF_Support_Projection_II_2}
v=\pi_{j}^{A}(v)=\pi_{j}^{A}\lc{}xw\rc{}
\end{align}

\noindent and

\begin{align}\label{EQ.LEM.AF_Support_Projection_II_3}
\chi_{(0,\infty]}(x)v=xw-\chi_{(0,\infty]}(x)\big(I-\pi_{j}^{A}\big)(x)w.
\end{align}


\pagebreak


Equation \ref{EQ.LEM.AF_Support_Projection_II_2} and Equation \ref{EQ.LEM.AF_Support_Projection_II_3} show

\begin{align}\label{EQ.LEM.AF_Support_Projection_II_4}
\pi_{j}^{A}\lc\chi_{(0,\infty]}(x)v\rc{}=v-\pi_{j}^{A}\lc\chi_{(0,\infty]}(x)\big(I-\pi_{j}^{A}\big)(x)w\rc{}.
\end{align}

\noindent Expanding $xw=x_{j}w+(I-\pi_{j}^{A})(x)w$ and using $A_{j}^{\perp}A_{j}\subset A_{j}^{\perp}$ for the second summand in the final term below, we apply Equation \ref{EQ.LEM.AF_Support_Projection_II_4} in order to estimate

\begin{align*}
\lgl\pi_{j}^{A}\lc\chi_{(0,\infty]}(x)v\rc{},v\rgl_{\tau} & =\| v\|_{\tau}^{2}-\lgl\big(I-\pi_{j}^{A}\big)(x)w,\chi_{(0,\infty]}(x)v\rgl_{\tau} \phantom{\bigg)} \\
& =\| v\|_{\tau}^{2}-\lgl\big(I-\pi_{j}^{A}\big)(x)w,xw-\chi_{(0,\infty]}(x)\big(I-\pi_{j}^{A}\big)(x)w\rgl_{\tau} \phantom{\bigg)} \\
& =\| v\|_{\tau}^{2}-\dblv{}\big(I-\pi_{j}^{A}\big)(x)w\dblv_{\tau}^{2}+\dblv{}\chi_{(0,\infty]}(x)\big(I-\pi_{j}^{A}\big)(x)w\dblv_{\tau}^{2} \phantom{\bigg)} \\
& \geq \| v\|_{\tau}^{2}-\dblv{}\big(I-\pi_{j}^{A}\big)(x)w\dblv_{\tau}^{2}. \phantom{\bigg)}
\end{align*}

\noindent Assume $v=u^{j}$ and $w=\sigma(x_{j})^{-1}u^{j}$. The above estimate yields

\begin{align}\label{EQ.LEM.AF_Support_Projection_II_5}
\dblv{}\chi_{(0,\infty]}(x)u^{j}\dblv_{\tau}^{2}=\lgl\pi_{j}^{A}\lc\chi_{(0,\infty]}(x)u^{j}\rc{},u^{j}\rgl_{\tau}\geq 1-\sigma(x_{j})^{-2}\dblv{}\big(I-\pi_{j}^{A}\big)(x)u^{j}\dblv_{\tau}^{2}.
\end{align}

\noindent We show $\dblv{}(I-\pi_{j}^{A})(x)u^{j}\dblv_{\tau}^{2}=0$. We calculate

\begin{align*}
\dblv{}\big(I-\pi_{j}^{A}\big)(x)u^{j}\dblv_{\tau}^{2} & = \dblv{}xu^{j}-x_{j}u^{j}\dblv_{\tau}^{2} \phantom{\bigg)} \\
& = \dblv{}xu^{j}\dblv_{\tau}^{2}-2\RE\lgl xu^{j},x_{j}u^{j}\rgl_{\tau}+\dblv{}x_{j}u^{j}\dblv_{\tau}^{2} \phantom{\bigg)} \\
& = \dblv{}xu^{j}\dblv_{\tau}^{2}-2\sigma(x_{j})\RE\lgl x_{j}u^{j},u^{j}\rgl_{\tau}+\sigma(x_{j})^{2} \phantom{\bigg)} \\
& = \dblv{}xu^{j}\dblv_{\tau}^{2}-\sigma(x_{j})^{2}. \phantom{\bigg)}
\end{align*}

\noindent Set $y:=x^{2}\in L^{1}(A,\tau)_{+}$. Then $\sqrt{y}=x$ by definition. Equation \ref{EQ.SSEC.QOT_AC_SUPP_4} for $z=y$ shows

\begin{align}\label{EQ.LEM.AF_Support_Projection_II_6}
\pi_{j}^{A}L_{y_{j}}\pi_{j}^{A}\leq L_{\pi_{j}^{A}(\sqrt{y})}^{2}=L_{x_{j}^{2}}.
\end{align}

\noindent Self-adjointness and Equation \ref{EQ.LEM.AF_Support_Projection_II_6} show

\begin{align}\label{EQ.LEM.AF_Support_Projection_II_7}
\dblv{}xu^{j}\dblv_{\tau}^{2}=\lgl yu^{j},u^{j}\rgl_{\tau}=\lgl y_{j}u^{j},u^{j}\rgl_{\tau}\leq \lgl x_{j}^{2}u^{j},u^{j}\rgl_{\tau}=\sigma(x_{j})^{2}.
\end{align}

\noindent Thus $0\leq \dblv{}(I-\pi_{j}^{A})(x)u^{j}\dblv_{\tau}^{2}=\dblv{}xu^{j}\dblv_{\tau}^{2}-\sigma(x_{j})^{2}\leq 0$, hence $\dblv{}(I-\pi_{j}^{A})(x)u^{j}\dblv_{\tau}^{2}=0$. Applying the latter to Equation \ref{EQ.LEM.AF_Support_Projection_II_5} yields

\begin{align}\label{EQ.LEM.AF_Support_Projection_II_8}
\dblv{}\chi_{(0,\infty]}(x)u^{j}\dblv_{\tau}^{2}=\lgl\pi_{j}^{A}\lc\chi_{(0,\infty]}(x)u^{j}\rc{},u^{j}\rgl_{\tau}\geq 1.
\end{align}


\pagebreak


For all $j\in\mathbb{N}$, choice of $u^{j}$ and Equation \ref{EQ.LEM.AF_Support_Projection_II_8} show

\begin{align}\label{EQ.LEM.AF_Support_Projection_II_9}
\sigma(x_{j})=\lgl x_{j}u^{j},u^{j}\rgl_{\tau}=\lgl xu^{j},u^{j}\rgl_{\tau}\geq\sigma(x)\dblv{}\chi_{(0,\infty]}(x)u^{j}\dblv_{\tau}^{2}\geq\sigma(x).
\end{align}

\noindent Equation \ref{EQ.LEM.AF_Support_Projection_II_9} implies $\limsup_{j\in\mathbb{N}}\sigma(x_{j})\geq\sigma(x)$. Then $1)$ follows as discussed above.\par
We show $2)$. Using $2)$ in Lemma \ref{LEM.Wstar_CLRA_FC}, note Equation \ref{EQ.SSEC.QOT_AC_SUPP_1} and Equation \ref{EQ.SSEC.QOT_AC_SUPP_2} ensure Lemma \ref{LEM.AF_Support_Projection_I} implies $2)$ if

\begin{align}\label{EQ.LEM.AF_Support_Projection_II_10}
\limsup_{j\in\mathbb{N}}\hspace{0.025cm} \dblv{}\chi_{(0,\infty]}\lc{}L_{x_{j}}\rc{}u\dblv_{\tau}=0
\end{align}

\noindent for all $u\in\im\ker L_{x}$. We reduce to $u\in\ker L_{x}\cap L^{\infty}(A,\tau)$. For all $v\in L^{2,\infty}(A,\tau)$, we see Equation \ref{EQ.SSEC.QOT_AC_SUPP_2} shows

\begin{align}\label{EQ.LEM.AF_Support_Projection_II_11}
\pi_{\ker L_{x}}^{A}(v)=L_{\delta_{0}(x)}(v)=\delta_{0}(x)v\in\ker L_{x}\cap L^{\infty}(A,\tau).
\end{align}

Note $2)$ in Proposition \ref{PRP.AF_Cstar_Trace_I} shows $A_{0}\subset L^{2,\infty}(A,\tau)\subset L^{2}(A,\tau)$ is $\|.\|_{\tau}$-dense. Let\linebreak  $u\in\ker L_{x}$ and fix arbitrary $\{u_{n}\}_{n\in\mathbb{N}}\subset L^{2,\infty}(A,\tau)$ s.t.~$u=\|.\|_{\tau}$-$\lim_{n\in\mathbb{N}}u_{n}$. For all $j\in\mathbb{N}$ and $n\in\mathbb{N}$, we have $\pi_{\ker x}^{A}\lc{}u_{n}\rc\in\ker L_{x}\cap L^{\infty}(A,\tau)$ by Equation \ref{EQ.LEM.AF_Support_Projection_II_11} and estimate

\begin{align}\label{EQ.LEM.AF_Support_Projection_II_12}
\dblv{}\chi_{(0,\infty]}(x_{j})u\dblv_{\tau}=\dblv{}\chi_{(0,\infty]}(x_{j})\pi_{\ker L_{x}}^{A}(u)\dblv_{\tau}\leq \dblv{}u-u_{n}\dblv_{\tau}+\dblv{}\chi_{(0,\infty]}(x_{j})\pi_{\ker L_{x}}^{A}\lc{}u_{n}\rc\dblv_{\tau}
\end{align}

\noindent as non-trivial projections have norm one. Equation \ref{EQ.LEM.AF_Support_Projection_II_12} implies Equation \ref{EQ.LEM.AF_Support_Projection_II_10} if

\begin{align}\label{EQ.LEM.AF_Support_Projection_II_13}
\limsup_{j\in\mathbb{N}}\hspace{0.025cm} \dblv{}\chi_{(0,\infty]}(x_{j})u\dblv_{\tau}=0
\end{align}

\noindent for all $u\in\ker L_{x}\cap L^{\infty}(A,\tau)$.\par
We show Equation \ref{EQ.LEM.AF_Support_Projection_II_13}. Assume $\tau\lc\chi_{(0,\infty]}(x)\rc{}<\infty$. Set $y:=x+\chi_{(0,\infty]}(x)\in L^{2}(A,\tau)_{+}$. Note $\chi_{(0,\infty]}(y)=\chi_{(0,\infty]}(x)$ and $\sigma(y)\geq 1$ by functional calculus. We know restriction maps are positivity-preserving by Proposition \ref{PRP.AF_Cstar_Trace_Dualisation_I}. For all $j\in\mathbb{N}$, get $x_{j}\leq y_{j}=x_{j}+\chi_{(0,\infty]}(x)_{j}$ and therefore 

\begin{align}\label{EQ.LEM.AF_Support_Projection_II_14}
\chi_{(0,\infty]}(x_{j})=\supp x_{j}\leq\supp y_{j}=\chi_{(0,\infty]}(y_{j})
\end{align}

\noindent by $2)$ in Proposition \ref{PRP.Support_Projection_I} and $1)$ in Proposition \ref{PRP.Support_Projection_II}. Note $2)$ in the latter proposition shows $\ker L_{x}=\ker L_{y}$ since we have $\chi_{(0,\infty]}(y)=\chi_{(0,\infty]}(x)$. For all $u\in\ker L_{x}\cap L^{\infty}(A,\tau)=\ker L_{y}\cap L^{\infty}(A,\tau)$ and $j\in\mathbb{N}$, we calculate

\begin{align*}
0=\lgl yu,u\rgl_{\tau} & = \lgl y_{j}u,u\rgl+\lgl\big(I-\pi_{j}^{A}\big)(y)u,u\rgl_{\tau} \phantom{\bigg)} \\
& \geq \sigma(y_{j})\dblv{}\chi_{(0,\infty]}(y_{j})u\dblv_{\tau}^{2}+\lgl\big(I-\pi_{j}^{A}\big)(y)u,u\rgl_{\tau}. \phantom{\bigg)}
\end{align*}

For all $j\in\mathbb{N}$, we have $\dblv{}\chi_{(0,\infty]}(x_{j})u\dblv_{\tau}^{2}\leq \dblv{}\chi_{(0,\infty]}(y_{j})u\dblv_{\tau}^{2}$ by Equation \ref{EQ.LEM.AF_Support_Projection_II_14} and further $\absv{1.15}{\lgl(I-\pi_{j}^{A})(y)u,u\rgl_{\tau}}\leq \dblv{}(I-\pi_{j}^{A})(x)\dblv_{\tau}\|u\|_{\infty}\|u\|_{\tau}<\infty$ by reducing to $\ker L_{x}\cap L^{\infty}(A,\tau)$. We also use $1)$ for $\lim_{j\in\mathbb{N}}\sigma(y_{j})=1$, and note $3)$ in Proposition \ref{PRP.AF_Cstar_Trace_III}. Altogether, we have

\begin{align}\label{EQ.LEM.AF_Support_Projection_II_15}
0=\lgl yu,u\rgl=\limsup_{j\in\mathbb{N}}\hspace{0.025cm} \lgl y_{j}u,u\rgl_{\tau}\geq\limsup_{j\in\mathbb{N}}\hspace{0.025cm} \dblv{}\chi_{(0,\infty]}(x_{j})u\dblv_{\tau}^{2}\geq 0
\end{align}

\noindent for all $u\in\ker x\cap L^{\infty}(A,\tau)$. Equation \ref{EQ.LEM.AF_Support_Projection_II_15} shows Equation \ref{EQ.LEM.AF_Support_Projection_II_13}. We see Equation \ref{EQ.LEM.AF_Support_Projection_II_10} and therefore $2)$ follows as discussed above.
\end{proof}

\begin{dfn}\label{DFN.AF_Support_Projection_II}
Let $x\in L^{1}(A,\tau)_{+}$. We say that $x$ has

\begin{itemize}
\item[1)] integrable support if $\tau\lc\supp x\rc{}<\infty$,

\item[2)] reducible support if $\supp x=\s$-$\lim_{j\in\mathbb{N}}\supp x_{j}$.
\end{itemize}
\end{dfn}

\begin{thm}\label{THM.AF_Support_Projection}
Let $(A,\tau)$ be a tracial AF-$C^{*}$-algebra. Let $x\in L^{1}(A,\tau)_{+}$. If we have $\tau\lc\supp x\rc{}<\infty$, then $\supp x=\s$-$\lim_{j\in\mathbb{N}}\supp x_{j}$.
\end{thm}
\begin{proof}
Theorem \ref{THM.AF_Cstar_Bimodule_CLRA_SR} states $L_{x}=\sr$-$\lim_{j\in\mathbb{N}}L_{x_{j}}$. Using $2)$ in Lemma \ref{LEM.Wstar_CLRA_FC}, as well as $1)$ and $2)$ in Proposition \ref{PRP.Support_Projection_II}, Equation \ref{EQ.SSEC.QOT_AC_SUPP_1} and Equation \ref{EQ.SSEC.QOT_AC_SUPP_2} show Lemma \ref{LEM.AF_Support_Projection_I} implies our claim if

\begin{align}\label{EQ.THM.AF_Support_Projection_1}
\limsup_{j\in\mathbb{N}}\hspace{0.025cm} \dblv{}\supp x_{j}\cdot u\dblv_{\tau}=0  
\end{align}

\noindent for all $u\in\ker L_{x}$. Note $1)$ in Proposition \ref{PRP.Support_Projection_II} shows $\supp x=\chi_{(0,\infty]}(\sqrt{x})$ by positivity and functional calculus. For all $j\in\mathbb{N}$, get $\supp x_{j}=\chi_{(0,\infty]}\lc\sqrt{x_{j}}\rc$. Equation \ref{EQ.SSEC.QOT_AC_SUPP_4} for $z=x$ further shows $\tau\lc\sqrt{x_{j}}p\rc{}=0$ for all projections $p\in A_{j}$ s.t.~$\tau\lc\pi_{j}^{A}(\sqrt{x})p\rc{}=0$.\par
For all $j\in\mathbb{N}$, $2)$ in Proposition \ref{PRP.Support_Projection_I} and $1)$ in Proposition \ref{PRP.Support_Projection_II} therefore show

\begin{align}\label{EQ.THM.AF_Support_Projection_2}
\chi_{(0,\infty]}\lc\sqrt{x_{j}}\rc\leq\chi_{(0,\infty]}\lc\pi_{j}^{A}(\sqrt{x})\rc{}.
\end{align}

\noindent Thus $1)$ in Proposition \ref{PRP.Support_Projection_II} and $2)$ in Lemma \ref{LEM.AF_Support_Projection_II} show

\begin{align}\label{EQ.THM.AF_Support_Projection_3}
\supp x=\s\textrm{-}\lim_{j\in\mathbb{N}}\hspace{0.025cm} \supp\pi_{j}^{A}(\sqrt{x})=\s\textrm{-}\lim_{j\in\mathbb{N}}\hspace{0.025cm} \chi_{(0,\infty]}\lc\pi_{j}^{A}(\sqrt{x})\rc{},
\end{align}

\noindent hence Equation \ref{EQ.THM.AF_Support_Projection_2} and Equation \ref{EQ.THM.AF_Support_Projection_3} let us estimate

\begin{align}\label{EQ.THM.AF_Support_Projection_4}
0=\dblv{}\supp x\cdot u\dblv_{\tau}^{2}=\lim_{j\in\mathbb{N}}\hspace{0.025cm} \dblv{}\chi_{(0,\infty]}\lc\pi_{j}^{A}(\sqrt{x})\rc\cdot u\dblv_{\tau}^{2}\geq\limsup_{j\in\mathbb{N}}\hspace{0.025cm} \dblv{}\supp x_{j}\cdot u\dblv_{\tau}^{2}\geq 0
\end{align}

\noindent for all $u\in\ker L_{x}$. Equation \ref{EQ.THM.AF_Support_Projection_4} immediately shows Equation \ref{EQ.THM.AF_Support_Projection_1}. We obtain our claim as described above.
\end{proof}

\begin{cor}\label{COR.AF_Support_Projection_I}
If $\tau<\infty$, then all $x\in L^{1}(A,\tau)_{+}$ have reducible support.
\end{cor}
\begin{proof}
Apply Theorem \ref{THM.AF_Support_Projection}. 
\end{proof}

\begin{cor}\label{COR.AF_Support_Projection_II}
If $x\in L^{1}(A,\tau)_{+}$ has spectral gap, then $x$ has reducible support.
\end{cor}
\begin{proof}
Note $1)$ in Proposition \ref{PRP.Support_Projection_II} shows $x\geq\sigma(x)\cdot \supp x$ by functional calculus. Thus $\tau\lc\supp x\rc\leq\sigma(x)^{-1}\tau(x)<\infty$ since $\sigma(x)>0$, hence Theorem \ref{THM.AF_Support_Projection} applies.
\end{proof}

Theorem \ref{THM.AF_Support_Projection} gives sufficient conditions for reducible support. Non-integrability does not exclude reducible support in general. All injective $x\in L^{1}(A,\tau)_{+}$ have reducible support by Lemma \ref{LEM.AF_Support_Projection_I} \lc{}cf.~Proposition \ref{PRP.FC_Injectivity}\rc{}. If $(A,\tau)=\lc\KII(H),\tr\rc$ for a separable Hilbert space $H$, then Example \ref{BSP.AF_Support_Projection} shows integrable support is equivalent to being a finite-dimensional matrix.

\begin{bsp}\label{BSP.AF_Support_Projection}
Let $H$ be a separable Hilbert space. Assume $(A,\tau)=\lc\KII(H),\tr\rc$. Let $x\in S_{1}(H)$. There exists $U\in\UII\lc\BII(H)\rc$ s.t.~$UxU^{*}$ has diagonal form. We know the latter is determined by $\lset\lambda_{n}\rset_{n\in\mathbb{N}}\subset [0,\infty)$ up to reordering. Applications of unitary conjugations are normal unital $^{*}$-homomorphisms. Thus $\supp x=U^{*}\lc\supp UxU^{*}\rc{}U$ by Lemma \ref{LEM.FC_Preservation_II} and Corollary \ref{COR.FC_Preservation}, hence $\tr\lc\supp x\rc{}=\tr\lc\supp UxU^{*}\rc$. Ergo $\tau\lc\supp x\rc{}<\infty$ if and only if $\lambda_{n}=0$ for a.e.~$n\in\mathbb{N}$, i.e.~$\tr\lc\supp x\rc{}<\infty$ if and only if $\supp x\in\KII(H)_{0}=\bigcup_{n\in\mathbb{N}}M_{n}(\mathbb{C})$.
\end{bsp}


\subsection{Noncommutative heat semigroups of quantum Laplacians}\label{SSEC.QOT_AC_HSG}

Noncommutative heat semigroups of quantum Laplacians are trace-preserving, as well as completely Markovian. In the finite-dimensional setting, self-adjointness implies quantum Laplacians satisfy, up to sign, a quantum Fokker-Planck equation with vanishing drift term \cite{BK.Gar_Zol.2004.Quantum_Noise}, i.e.~only diffusion term. The latter solve special cases of general Lindblad master equations \cite{BK.Dav.1976.Quantum_Markov_SG}\cite{BK.Gar_Zol.2004.Quantum_Noise}\cite{ART.Spo.1980.Physics_Markov_SG_Master_Equation} describing purely irreversible time-evolution of dissipative quantum systems \cite{BK.Bra.1987.OpAlg_Quantum_StM_I}\cite{BK.Bra.1987.OpAlg_Quantum_StM_II}\cite{BK.Dav.1976.Quantum_Markov_SG}\cite{BK.Gar_Zol.2004.Quantum_Noise}\cite{BK.Ohy_Pet.1993.Rel_Ent}\cite{BK.Ste_vLee.2013.Full_Quantum_StM}. Of course, the sign occurs since negatives of quantum Laplacians generate noncommutative heat semigroups.\par
We view such diffusion terms of quantum Fokker-Planck equations as infinitesimal applications of quantum channels \cite{ART.Bih_Lid_Wha.2001.Quantum_Markov_CPM_Reconstruction}\cite{ART.Cub_Eis_Wol.2012.Quantum_Markov_SG_Reconstruction} transmitting change of states of the given quantum system determined by irreversible interaction with its environment \cite{BK.Nie_Chu.2000.Quantum_Computation_Information}\cite{ART.Kra.1971.State_Changes}. The extension \cite{BK.Cam_Def.2019.Quantum_StM_Information}\cite{ART.DiVi_Loss.1998.Quantum_Information_Physical} of Landauer's principle \cite{ART.Lan.1961.Information_Physical_I}\cite{ART.Lan.1961.Information_Physical_II} gives strictly positive lower bounds on production of quantum entropy upon application of quantum channels due to minimal heat dissipation \cite{ART.Ara_Ber_Cil_Dil_Lut_Pet.2012.Information_Physical_Experimental_Verification}\cite{ART.Burz_Gau_Lui_Mae_vandZan.2018.Quantum_Information_Physical_Experimental_Verification}\cite{ART.Sag_Ueda.2008.Quantum_Information_Physical_From_2nd_Law_ThDy}. Following a maximum entropy production principle \cite{ART.Dew.2003.MaxEnt_Information_I}\cite{ART.Dew.2005.MaxEnt_Information_II}\cite{ART.Mar_Sel.2006.MaxEnt_Review}, we select noise diffusion terms in the finite-dimensional setting by maximising production of quantum entropy under constraints on energy spent and assume stability under scaling limits. Following our discussion of the coarse graining process in Subsection \ref{SSEC.QOT_CG}, we show quantum Laplacians satisfy, up to sign, a quantum Fokker-Planck equation with vanishing drift term in scaling limit, i.e.~only noise diffusion term. Altogether, we therefore view quantum Laplacians as generators of quantum noise evolution in Subsection \ref{SSEC.L2W_Log_Mean_QNE}, and obtain a description of quantum Laplacians in terms of both quantum statistical mechanics \cite{BK.Bra.1987.OpAlg_Quantum_StM_I}\cite{BK.Bra.1987.OpAlg_Quantum_StM_II} and quantum information theory \cite{BK.Nie_Chu.2000.Quantum_Computation_Information} as claimed in the introduction of this chapter.\par
We require regularisation of normal states under heat flow. Theorem \ref{THM.Wstar_Derivation_QG_HSG_Regularity} shows such regularisation by combining compressing with support projections of normal fixed states and finite-dimensional approximation. This uses compatibility with compression and finite-dimensional approximation. As such, each step of the coarse graining process terminates at accessibility components in the finite-dimensional setting s.t.~heat flow maps to their relative interiors for all non-zero times. Standard references for quantum statistical mechanics are \cite{BK.Bra.1987.OpAlg_Quantum_StM_I}\cite{BK.Bra.1987.OpAlg_Quantum_StM_II}, \cite{BK.Dav.1976.Quantum_Markov_SG}, \cite{BK.Gar_Zol.2004.Quantum_Noise}, \cite{BK.Ohy_Pet.1993.Rel_Ent} and \cite{BK.Ste_vLee.2013.Full_Quantum_StM}. Standard reference for quantum in\-formation theory is \cite{BK.Nie_Chu.2000.Quantum_Computation_Information}. We further use and refer to \cite{BK.Cam_Def.2019.Quantum_StM_Information} as comprehensive treatment of the quantum statistical mechanics of quantum information.


\subsubsection*{Completely Markovian semigroups}

We discuss both completely Markovian semigroups and Lindblad master equations, as well as their special case of quantum Fokker-Planck equations. Generalising the uniformly continuous case \cite{ART.Chr_Eva.1979.Noncommutative_Markov_SG_Lindblad}\cite{ART.Lin.1976.Quantum_Markov_SG_Master_Equation_I}\cite{ART.Lin.1976.Quantum_Markov_SG_Master_Equation_II} applied to open quantum systems \cite{ART.Dav.1974.Quantum_Markov_SG_Master_Equation_I}\cite{ART.Dav.1975.Quantum_Markov_SG_Master_Equation_III}\cite{ART.Dav.1976.Quantum_Markov_SG_Master_Equation_II}\cite{ART.Fri_Gor_Kos_Sud_Ver.1978.Quantum_Markov_SG_Master_Equation}, completely Markovian semigroups \cite{ART.Dav.1979.Quantum_Markov_SG_Generators}\cite{ART.Dav_Lin.1992.Noncommutative_Markov_SG_I}\cite{ART.Dav_Lin.1993.Noncommutative_Markov_SG_II} describe time-evolution of dissipative quantum systems weakly coupled to a heat bath \cite{BK.Dav.1976.Quantum_Markov_SG}\cite{BK.Gar_Zol.2004.Quantum_Noise}\cite{ART.Spo.1980.Physics_Markov_SG_Master_Equation}. Symmetric $C^{*}$-derivations are noncommutative gradients and define Laplacians generating completely Markovian noncommutative heat semigroups \cite{ART.Cip.1997.NC_Dirichlet_Markov}\cite{ART.Cip_Sav.2003.NC_Dirichlet_Grad}. Following Remark \ref{REM.CWstar_Derivation}, we specialise to the AF-$C^{*}$-setting in order to study noncommutative heat semigroups of quantum Laplacians.\par
Definition \ref{DFN.Wstar_CP_Markovian_SG} gives completely Markovian semigroups for tracial $C^{*}$-algebras. We use completely positive and completely Markovian maps \lc{}cf.~Definition \ref{DFN.Wstar_CP} and Definition \ref{DFN.Wstar_CP_Markovian}\rc{}. Lemma \ref{LEM.Wstar_CP_Markovian_SG} gives sufficient conditions for satisfying Equation \ref{EQ.LEM.Wstar_CP_Markovian_SG_1} as special case of general Lindblad master equations \cite{BK.Dav.1976.Quantum_Markov_SG}\cite{BK.Gar_Zol.2004.Quantum_Noise}\cite{ART.Spo.1980.Physics_Markov_SG_Master_Equation}. This yields Lindblad decompositions as per Definition \ref{DFN.Wstar_CP_Markovian_SG_II}. Following Remark \ref{REM.Wstar_CP_Markovian_SG}, Equation \ref{EQ.LEM.Wstar_CP_Markovian_SG_1} is a quantum Fokker-Planck equation with drift and diffusion terms as per Equation \ref{EQ.REM.Wstar_CP_Markovian_SG_1}. We view such diffusion terms as infinitesimal applications of quantum channels \cite{ART.Bih_Lid_Wha.2001.Quantum_Markov_CPM_Reconstruction}\cite{ART.Cub_Eis_Wol.2012.Quantum_Markov_SG_Reconstruction} transmitting change of states of the given quantum system determined by irreversible interaction with its environment \cite{BK.Nie_Chu.2000.Quantum_Computation_Information}\cite{ART.Kra.1971.State_Changes}. If self-adjointness in the finite-dimensional setting is given, then Corollary \ref{COR.Wstar_CP_Markovian_SG} shows we may assume vanishing drift term.\par
Let $(A,\tau)$ be a tracial $C^{*}$-algebra.

\begin{dfn}\label{DFN.Wstar_CP_Markovian_SG}
A semigroup $G:[0,\infty)\longrightarrow\BII\lc{}L^{\infty}(A,\tau)\rc$ is completely Markovian if $G_{t}:L^{\infty}(A,\tau)\longrightarrow L^{\infty}(A,\tau)$ is a completely Markovian normal map for all $t\geq 0$.
\end{dfn}

\begin{lem}\label{LEM.Wstar_CP_Markovian_SG}
Assume $\tau<\infty$. Let $S\in\BII\lc{}L^{2}(A,\tau)\rc_{h}$ s.t.~$S\neq 0$, $S\lc{}L^{\infty}(A,\tau)\rc\subset L^{\infty}(A,\tau)$ and $S(1_{A})=0$. We have semigroup $G^{S}:[0,\infty)\longrightarrow\BII\lc{}L^{\infty}(A,\tau)\rc$ by setting $G_{t}^{S}:=e^{tS}$ for all $t\geq 0$. If $S:L^{\infty}(A,\tau)\longrightarrow L^{\infty}(A,\tau)$ is normal and $G^{S}:[0,\infty)\longrightarrow\BII\lc{}L^{\infty}(A,\tau)\rc$ is a completely Markovian semigroup, then there exists $H\in L^{\infty}(A,\tau)_{h}$, completely positive normal $\varphi:L^{\infty}(A,\tau)\longrightarrow L^{\infty}(A,\tau)$ with $\dblv{}\varphi(1_{A})\dblv_{\infty}=1$, and $C>0$ satisfying the Lindblad master equation

\begin{align}\label{EQ.LEM.Wstar_CP_Markovian_SG_1}
S(x)=i[H,x]+\frac{C}{2}\bigg(2\varphi(x)-\big\{\varphi(1_{A}),x\big\}\bigg)
\end{align}

\noindent for all $x\in L^{\infty}(A,\tau)$.
\end{lem}

\begin{proof}
Note $G^{S}:[0,\infty)\longrightarrow\BII\lc{}L^{\infty}(A,\tau)\rc$ is a semigroup by boundedness and functional calculus. Assume $S:L^{\infty}(A,\tau)\longrightarrow L^{\infty}(A,\tau)$ is normal and $G^{S}:[0,\infty)\longrightarrow\BII\lc{}L^{\infty}(A,\tau)\rc$ is a completely Markovian semigroup. Set $\AII:=L_{L^{\infty}(A,\tau)}(A)$. Then $\AII''=L_{L^{\infty}(A,\tau)}\lc{}L^{\infty}(A,\tau)\rc$ is the $\sigma$-weak closure \lc{}cf.~Proposition \ref{PRP.Wstar_Equivalence} and Proposition \ref{PRP.Wstar_Trace_Ext_II}\rc{}. Theorem 3.1 in \cite{ART.Chr_Eva.1979.Noncommutative_Markov_SG_Lindblad} applies to the canonical lift of $G^{S}$ to $\AII''$. For all $t\geq 0$, set

\begin{align}\label{EQ.LEM.Wstar_CP_Markovian_SG_2}
S^{\dagger}:=L_{L^{\infty}(A,\tau)}\circ S\circ L_{L^{\infty}(A,\tau)}^{-1},\ G_{t}^{S,\dagger}:=L_{L^{\infty}(A,\tau)}\circ G_{t}^{S}\circ L_{L^{\infty}(A,\tau)}^{-1}.
\end{align}

\noindent We have $G_{t}^{S,\dagger}=e^{tS^{\dagger}}$ in each case by norm differentiation. Since $^{*}$-homomorphisms are completely positive \lc{}cf.~Example \ref{BSP.Wstar_CP_II}\rc{}, conjugation with canonical left-actions as per Equation \ref{EQ.LEM.Wstar_CP_Markovian_SG_2} preserves complete positivity. Moreover, normality is preserved by the GNS-construction \lc{}cf.~Proposition \ref{PRP.Wstar_Trace_Ext_II}\rc{}. Thus $G^{S,\dagger}:[0,\infty)\longrightarrow\BII\lc\AII''\rc$ is a uniformly $\|.\|_{\AII''}$-continuous semigroup s.t.~$G_{t}^{S,\dagger}:\AII''\longrightarrow\AII''$ is a completely Markovian normal map for all $t\geq 0$, hence Theorem 3.1 in \cite{ART.Chr_Eva.1979.Noncommutative_Markov_SG_Lindblad} applies.\par
We apply Theorem 3.1 in \cite{ART.Chr_Eva.1979.Noncommutative_Markov_SG_Lindblad}. The theorem yields $H^{\dagger}\in\AII_{h}''$ and completely positive $\varphi^{\dagger}:\AII''\longrightarrow\AII''$ s.t.~

\begin{align}\label{EQ.LEM.Wstar_CP_Markovian_SG_3}
S^{\dagger}\lc{}L_{x,L^{\infty}(A,\tau)}\rc{}=i\lb{}H^{\dagger},L_{x,L^{\infty}(A,\tau)}\rb{}+\varphi^{\dagger}\lc{}L_{x,L^{\infty}(A,\tau)}\rc{}-\frac{1}{2}\big\{\varphi^{\dagger}(I),L_{x,L^{\infty}(A,\tau)}\big\}
\end{align}

\noindent for all $x\in L^{\infty}(A,\tau)$. Using $S^{\dagger}:\AII''\longrightarrow\AII''$ normal, Equation \ref{EQ.LEM.Wstar_CP_Markovian_SG_3} implies $\varphi^{\dagger}:\AII''\longrightarrow\AII''$ is normal by rearranging terms accordingly. Set

\begin{align}\label{EQ.LEM.Wstar_CP_Markovian_SG_4}
H:=L_{L^{\infty}(A,\tau)}^{-1}\circ H^{\dagger}\circ L_{L^{\infty}(A,\tau)},\ \varphi^{S}:=L_{L^{\infty}(A,\tau)}^{-1}\circ\varphi^{\dagger}\circ L_{L^{\infty}(A,\tau)}.
\end{align}

\noindent Since we conjugate with normal $^{*}$-homomorphisms, get $H\in L^{\infty}(A,\tau)_{h}$ and completely positive normal $\varphi^{S}:L^{\infty}(A,\tau)\longrightarrow L^{\infty}(A,\tau)$. Using the latter, applying Equation \ref{EQ.LEM.Wstar_CP_Markovian_SG_2} and Equation \ref{EQ.LEM.Wstar_CP_Markovian_SG_4} to Equation \ref{EQ.LEM.Wstar_CP_Markovian_SG_3} shows

\begin{align}\label{EQ.LEM.Wstar_CP_Markovian_SG_5}
S(x)=i[H,x]+\frac{1}{2}\bigg(2\varphi^{S}(x)-\big\{\varphi^{S}(1_{A}),x\big\}\bigg)
\end{align}

\noindent for all $x\in L^{\infty}(A,\tau)$. Note $\varphi^{S}(1_{A})=0$ implies $\varphi^{S}=0$ by positivity-preservation. Since $H\in L^{\infty}(A,\tau)_{h}$, as well as $S\in\BII\lc{}L^{2}(A,\tau)\rc_{h}$ and $S\neq 0$, Equation \ref{EQ.LEM.Wstar_CP_Markovian_SG_5} shows $\varphi^{S}(1_{A})\neq 0$ by self-adjointness. Equation \ref{EQ.LEM.Wstar_CP_Markovian_SG_5} therefore shows

\begin{align}\label{EQ.LEM.Wstar_CP_Markovian_SG_6}
H,\ \varphi:=\dblv{}\varphi^{S}(1_{A})\dblv_{\infty}^{-1}\varphi^{S},\ C:=\dblv{}\varphi^{S}(1_{A})\dblv_{\infty}
\end{align}

\noindent satisfy Equation \ref{EQ.LEM.Wstar_CP_Markovian_SG_1} for $S$ as claimed.
\end{proof}


\pagebreak


\begin{dfn}\label{DFN.Wstar_CP_Markovian_SG_II}
Assume the setting of Lemma \ref{LEM.Wstar_CP_Markovian_SG}.

\begin{itemize}
\item[1)] We call $G^{S}:[0,\infty)\longrightarrow\BII\lc{}L^{\infty}(A,\tau)\rc$ the induced semigroup of $S$.

\item[2)] If $\restr{0.925}{S}{L^{\infty}(A,\tau)}:L^{\infty}(A,\tau)\longrightarrow L^{\infty}(A,\tau)$ is normal and $G^{S}$ completely Markovian, then we call $(H,\varphi,C)$ as per Equation \ref{EQ.LEM.Wstar_CP_Markovian_SG_1} a Lindblad decomposition of $S$.
\end{itemize}
\end{dfn}

\begin{cor}\label{COR.Wstar_CP_Markovian_SG}
Assume $A$ is finite-dimensional. Let $S\in\BII(A)_{h}$ s.t.~$S\neq 0$ and $S(1_{A})=0$. If $S$ has completely Markovian induced semigroup, then there exists completely positive self-adjoint normal $\varphi:A\longrightarrow A$ with $\dblv{}\varphi(1_{A})\dblv_{\infty}=1$ and $C>0$ s.t.~$\lc{}0,\varphi,C\rc$ is a Lindblad decomposition of $S$.
\end{cor}
\begin{proof}
We have finite-dimensional tracial $W^{*}$-algebra $\lc\BII(A),\tr\rc$. For all $T\in\BII(A)$, we decompose $T=\RE(T)+i\IM(T)$ into real and imaginary parts

\begin{align}\label{EQ.COR.Wstar_CP_Markovian_SG_1}
\RE(T)=\frac{T+T^{*}}{2},\ \IM(T)=-i\frac{T-T^{*}}{2}
\end{align}

\noindent as per $1)$ in Proposition \ref{PRP.Wstar_NCI_IV}. Equation \ref{EQ.COR.Wstar_CP_Markovian_SG_1} yields $\BII(A)=\BII(A)_{h}\oplus\BII(A)_{h}$ using direct sum of real vector spaces. Let $T\in\BII(A)$. We have $T\in\BII(A)_{h}$ if and only if $\IM(T)=0$. For all $u,v\in A$, set $x:=v^{*}v,y:=uu^{*}\in A_{+}$ and calculate

\begin{align}\label{EQ.COR.Wstar_CP_Markovian_SG_2}
\lgl L_{T(x)}u,u\rgl_{\tau}=\lgl v^{*}v,T^{*}(y)\rgl_{\tau}=\lgl v,vT^{*}(y)\rgl_{\tau}=\lgl v,R_{T^{*}(y)}(v)\rgl_{\tau}=\lgl R_{\lc{}T^{*}(y)\rc^{*}}(v),v\rgl_{\tau}.
\end{align}

\noindent For all $y\in A_{+}$, we have $T^{*}(y)\geq 0$ if and only if $\lc{}T^{*}(y)\rc^{*}\geq 0$ since $A_{+}\subset A_{h}$. Using the latter and $3)$ in Proposition \ref{PRP.Wstar_CLRA_FC}, Equation \ref{EQ.COR.Wstar_CP_Markovian_SG_2} implies $T$ is positivity-preserving if and only if $T^{*}$ is. For all $n\in\mathbb{N}$, we argue analogously upon replacing $(A,\tau)$ with the finite-dimensional tracial $C^{*}$-algebra $\lc{}A\otimes M_{n}(\mathbb{C}),\tau\otimes\tr_{n}\rc$. Altogether, we know $T$ is completely positive if and only if $T^{*}$ is. We may also use Proposition \ref{PRP.AF_Cstar_Trace_II} and reduce to Choi's theorem \cite{BK.Dav.1976.Quantum_Markov_SG} for pairs of summands in $A\cong\oplus_{l=1}^{n}M_{n_{l}}(\mathbb{C})$, i.e.~representations as per Equation \ref{EQ.SSEC.QOT_AC_HSG_1} up to conjugation with projections, for alternative proof. If $T\in\BII(A)_{h}$ is completely positive, then the first identity in Equation \ref{EQ.COR.Wstar_CP_Markovian_SG_1} shows $\RE(T)$ is completely positive, and the second one $\RE(T)(1_{A})=T(1_{A})$ since $T(1_{A})=T^{*}(1_{A})$ by $\IM(T)=0$.\par
Normality is equivalent to boundedness in the finite-dimensional setting. Assume $S$ has completely Markovian induced semigroup. We are in the setting of Lemma \ref{LEM.Wstar_CP_Markovian_SG}. Let $(H,\varphi,C)$ be a Lindblad decomposition of $S$. Note $[H,\blank]\in\BII(A)_{h}$. Using the latter and $S\in\BII(A)_{h}$, Equation \ref{EQ.LEM.Wstar_CP_Markovian_SG_1} and Equation \ref{EQ.COR.Wstar_CP_Markovian_SG_1} show

\begin{align}\label{EQ.COR.Wstar_CP_Markovian_SG_3}
S=\RE(S)=\frac{C}{2}\bigg(2\RE(\varphi)-\big\{\varphi(1_{A}),\blank\big\}\bigg),\ 0=\IM(S)=i[H,\blank]+iC\IM(\varphi).
\end{align}

\noindent Using $\lb{}H,1_{A}\rb{}=0$, the second identity in Equation \ref{EQ.COR.Wstar_CP_Markovian_SG_3} shows $\IM(S)(1_{A})=0$ at once and therefore $\RE(\varphi)(1_{A})=\varphi(1_{A})$. Thus $\dblv{}\RE(\varphi)(1_{A})\dblv_{\infty}=\dblv{}\varphi(1_{A})\dblv_{\infty}=1$, hence we have completely positive self-adjoint normal $\RE(\varphi):A\longrightarrow A$ with $\dblv{}\RE(\varphi)(1_{A})\dblv_{\infty}=1$ since $\varphi:A\longrightarrow A$ is completely positive normal with $\dblv{}\varphi(1_{A})\dblv_{\infty}=1$ by hypothesis. The first identity in Equation \ref{EQ.COR.Wstar_CP_Markovian_SG_3} shows $\lc{}0,\RE(\varphi),C\rc$ is Lindblad decomposition of $S$.
\end{proof}

We show Equation \ref{EQ.LEM.Wstar_CP_Markovian_SG_1} is a special case of a general Lindblad master equation \lc{}cf.~Equation 5.2.29 in \cite{BK.Gar_Zol.2004.Quantum_Noise}\rc{}. Assume the setting of Lemma \ref{LEM.Wstar_CP_Markovian_SG}. We use notation from its proof. Assume $A$ is separable. Note $\tau<\infty$ ensures $L^{2}(A,\tau)$ is separable.\par
Let $(H,\varphi,C)$ be a Lindblad decomposition of $S$. Upon conjugation with canonical left-actions as per Equation \ref{EQ.LEM.Wstar_CP_Markovian_SG_2}, Theorem 3.1 in \cite{ART.Chr_Eva.1979.Noncommutative_Markov_SG_Lindblad} yields Lindblad decomposition $\lc{}H^{\dagger},\varphi^{\dagger},C\rc$ of $S^{\dagger}$. Using separability of $L^{2}(A,\tau)$ in order to have a sequence, Theorem 2.3 in Chapter 9 in \cite{BK.Dav.1976.Quantum_Markov_SG} shows there exist $\lset{}W_{n}\rset_{n\in\mathbb{N}}\subset\BII\lc{}L^{2}(A,\tau)\rc$ s.t.~we have $\sum_{n\in\mathbb{N}}\mathrlap{\phantom{W}_{n}}W^{*}TW_{n}=\w$-$\lim_{m\in\mathbb{N}}\sum_{n=1}^{m}\mathrlap{\phantom{W}_{n}}W^{*}TW_{n}$ and further

\begin{align}\label{EQ.SSEC.QOT_AC_HSG_1}
\varphi^{\dagger}(T)=\sum_{n\in\mathbb{N}}\mathrlap{\phantom{W}_{n}}W^{*}TW_{n}
\end{align}

\noindent for all $T\in\AII''$. Using unitality of canonical left-actions of tracial $W^{*}$-algebras, we have $\sum_{n\in\mathbb{N}}\mathrlap{\phantom{W}_{n}}W^{*}W_{n}\leq I$ since $\varphi^{\dagger}:\AII''\longrightarrow\AII''$ is positivity-preserving with $\dblv{}\varphi^{\dagger}(I)\dblv_{\infty}=1$. This lets us relax unitality $\sum_{n\in\mathbb{N}}\mathrlap{\phantom{W}_{n}}W^{*}W_{n}=I$ in the definition of quantum channels \cite{BK.Nie_Chu.2000.Quantum_Computation_Information}\cite{ART.Kra.1971.State_Changes}.\par
Equation \ref{EQ.SSEC.QOT_AC_HSG_1} is a Kraus operator representation of $\varphi^{\dagger}$ with $\lset{}W_{n}\rset_{n\in\mathbb{N}}\subset\BII\lc{}L^{2}(A,\tau)\rc$ its Kraus operators \cite{ART.Kra.1971.State_Changes}. Applying Equation \ref{EQ.SSEC.QOT_AC_HSG_1} to Equation \ref{EQ.LEM.Wstar_CP_Markovian_SG_1} for $S^{\dagger}$ yields

\begin{align}\label{EQ.SSEC.QOT_AC_HSG_2}
S^{\dagger}(T)=i\lb{}H^{\dagger},T\rb{}+\sum_{n\in\mathbb{N}}\mathrlap{\phantom{W}_{n}}W^{*}TW_{n}-\frac{1}{2}\big\{\mathrlap{\phantom{W}_{n}}W^{*}W_{n},T\big\}
\end{align}

\noindent for all $T\in\AII''$. Pulled-back along the canonical left-action, we have

\begin{align}\label{EQ.SSEC.QOT_AC_HSG_3}
\varphi(x)=L_{L^{\infty}(A,\tau)}^{-1}\vstretch{1.1875}{\Bigg(}\sum_{n\in\mathbb{N}}\mathrlap{\phantom{W}_{n}}W^{*}L_{x,L^{\infty}(A,\tau)}W_{n}\vstretch{1.1875}{\Bigg)}
\end{align}

\noindent for all $x\in L^{\infty}(A,\tau)$. Equation \ref{EQ.SSEC.QOT_AC_HSG_2} and Equation \ref{EQ.SSEC.QOT_AC_HSG_3} show

\begin{align}\label{EQ.SSEC.QOT_AC_HSG_4}
S(x)=i[H,x]+\frac{C}{2}L_{L^{\infty}(A,\tau)}^{-1}\vstretch{1.1875}{\Bigg(}\sum_{n\in\mathbb{N}}\mathrlap{\phantom{W}_{n}}W^{*}L_{x,L^{\infty}(A,\tau)}W_{n}-\frac{1}{2}\big\{\mathrlap{\phantom{W}_{n}}W^{*}W_{n},L_{x,L^{\infty}(A,\tau)}\big\}\vstretch{1.1875}{\Bigg)}
\end{align}

\noindent for all $x\in L^{\infty}(A,\tau)$. Equation \ref{EQ.SSEC.QOT_AC_HSG_4} is called a Kraus operator representation of $S$ and Equation \ref{EQ.LEM.Wstar_CP_Markovian_SG_1}. Up to strictly positive constants, Equation \ref{EQ.SSEC.QOT_AC_HSG_2}, i.e.~Equation \ref{EQ.LEM.Wstar_CP_Markovian_SG_1} via Kraus operator representation as per Equation \ref{EQ.SSEC.QOT_AC_HSG_4}, is a general Lindblad master equation. Following Remark \ref{REM.Wstar_CP_Markovian_SG}, we additionally know Equation \ref{EQ.LEM.Wstar_CP_Markovian_SG_1} is a quantum Fokker-Planck equation with drift and diffusion terms as per Equation \ref{EQ.REM.Wstar_CP_Markovian_SG_1} s.t.~their diffusion terms are infinitesimal applications of quantum channels.

\begin{rem}\label{REM.Wstar_CP_Markovian_SG}
Note general Lindblad master equations \lc{}cf.~Equation 5.2.29 in \cite{BK.Gar_Zol.2004.Quantum_Noise}\rc{} specialise to quantum Fokker-Planck equations as follows. If quantum white noise is the input for a given quantum system, then its associated quantum Langevin equation \lc{}cf.~Equation 5.3.15 in \cite{BK.Gar_Zol.2004.Quantum_Noise}\rc{} determines a quantum stochastic differential equation in It\^o form \lc{}cf.~Equation 5.3.50 in \cite{BK.Gar_Zol.2004.Quantum_Noise}\rc{} based on a quantum Wiener process.\par


\pagebreak


Using reduced trace obtained by the weak coupling assumption, dualisation yields a linear differential equation of density operators \lc{}cf.~Equation 5.4.12 in \cite{BK.Gar_Zol.2004.Quantum_Noise}\rc{}. It is a quantum Fokker-Planck equation describing time-evolution of the given quantum system under quantum white noise similar to the classical case \cite{BK.Ris.1989.Fokker_Planck}. Indeed, it is a general Lindblad master equation s.t.~commutators are taken w.r.t.~the Hamiltonian of the given quantum system, and separates into distinct drift and diffusion terms arising from corresponding terms with physical meaning in the quantum Langevin equation. The former arise from all reversible interactions within quantum systems, whereas the latter do from all irreversible ones with their given environments. For details on general Lindblad master equations and the above derivation, we refer to Chapter 5 in \cite{BK.Gar_Zol.2004.Quantum_Noise}. For details on their many applications, we refer to \cite{BK.Dav.1976.Quantum_Markov_SG}, \cite{BK.Gar_Zol.2004.Quantum_Noise} and \cite{ART.Spo.1980.Physics_Markov_SG_Master_Equation}.\par
Assume the setting of Lemma \ref{LEM.Wstar_CP_Markovian_SG}. Let $(H,\varphi,C)$ be a Lindblad decomposition of $S$. We consider $H$ as Hamiltonian of a quantum system. Using the latter and following our above discussion, note Equation \ref{EQ.LEM.Wstar_CP_Markovian_SG_1} is a quantum Fokker-Planck equation s.t.~its commutator is taken w.r.t.~$H$. We have drift term $S^{\textrm{Drift}}\in i\BII\lc{}L^{2}(A,\tau)\rc_{h}$ and diffusion term $S^{\textrm{Diff}}\in\BII\lc{}L^{2}(A,\tau)\rc_{h}$ given by

\begin{align}\label{EQ.REM.Wstar_CP_Markovian_SG_1}
S^{\textrm{Drift}}(x)=i[H,x],\ S^{\textrm{Diff}}(x)=\frac{C}{2}\cdot \bigg(2\varphi(x)-\big\{\varphi(1_{A}),x\big\}\bigg)
\end{align}

\noindent for all $x\in L^{\infty}(A,\tau)$. Following our above discussion, $S^{\textrm{Drift}}$ is the reversible part, and $S^{\textrm{Diff}}$ the irreversible part of Equation \ref{EQ.LEM.Wstar_CP_Markovian_SG_1}. Altogether, Equation \ref{EQ.LEM.Wstar_CP_Markovian_SG_1} is described in terms of quantum statistical mechanics \cite{BK.Bra.1987.OpAlg_Quantum_StM_I}\cite{BK.Bra.1987.OpAlg_Quantum_StM_II}. We view $S^{\textrm{Diff}}$ as infinitesimal application of the quantum channel $\varphi:L^{\infty}(A,\tau)\longrightarrow L^{\infty}(A,\tau)$ below. If $H=0$, then we thereby describe Equation \ref{EQ.LEM.Wstar_CP_Markovian_SG_1} in terms of quantum information theory \cite{BK.Nie_Chu.2000.Quantum_Computation_Information}.\par
Completely positive normal unital maps are quantum channels \lc{}cf.~pp.353-373 in \cite{BK.Nie_Chu.2000.Quantum_Computation_Information}\rc{}. We may relax unitality in Kraus operator representations \lc{}cf.~p.360 in \cite{BK.Nie_Chu.2000.Quantum_Computation_Information}\rc{}. Each quantum channel describes a state change due to measurement \lc{}cf.~pp.360-364 in \cite{BK.Nie_Chu.2000.Quantum_Computation_Information} or \cite{ART.Dav_Lew.1970.Wstar_Quantum_Probability}\cite{ART.Kra.1971.State_Changes}\cite{BK.Ohy_Pet.1993.Rel_Ent}\rc{}, i.e.~each transmits a corresponding change of information encoded in states of the given quantum system \lc{}cf.~365-373 in \cite{BK.Nie_Chu.2000.Quantum_Computation_Information}\rc{} providing physical realisation of a quantum computer \lc{}cf.~Chapter 7 in \cite{BK.Nie_Chu.2000.Quantum_Computation_Information} or \cite{ART.Ash_Geo_Nor.2014.Quantum_Simulation}\cite{ART.Bur_Lad_Nic.2023.QC_Spin_Overview}\rc{}. We therefore have quantum channel $\varphi:L^{\infty}(A,\tau)\longrightarrow L^{\infty}(A,\tau)$. The second identity in Equation \ref{EQ.REM.Wstar_CP_Markovian_SG_1} shows

\begin{align}\label{EQ.REM.Wstar_CP_Markovian_SG_2}
S^{\textrm{Diff}}(x)=C\cdot \Bigg(\lc\varphi(x)-x\rc{}-\Bigg[\frac{1}{2}\big\{\varphi(1_{A}),x\big\}-x\Bigg]\Bigg)
\end{align}

\noindent for all $x\in L^{\infty}(A,\tau)$. If $\varphi$ is unital, then the second term in Equation \ref{EQ.REM.Wstar_CP_Markovian_SG_2} vanishes. Up to strictly positive constant, Equation \ref{EQ.REM.Wstar_CP_Markovian_SG_2} shows $S^{\textrm{Diff}}$ is the difference operator given by $\varphi$ minus a correction term controlling for non-unitality. The latter uses anti-commutator given by the arithmetic operator mean for two variables evaluated on $\varphi(1_{A})$ \cite{ART.And_Kub.1979.Operator_Means}. It is a quantum channel and the correction terms its difference operator. Up to energy scale but accounting for non-unitality, Equation \ref{EQ.REM.Wstar_CP_Markovian_SG_2} shows $\varphi$ transmits change of states of the given quantum system arising from irreversible interactions with its environment as per $S^{\textrm{Diff}}$ for a discrete time-step, resp.~applying $S^{\textrm{Diff}}$ yields such change as per $\varphi$ but infinitesimally \cite{ART.Bih_Lid_Wha.2001.Quantum_Markov_CPM_Reconstruction}\cite{ART.Cub_Eis_Wol.2012.Quantum_Markov_SG_Reconstruction}. We therefore view $S^{\textrm{Diff}}$ as infinitesimal application of $\varphi$.
\end{rem}


\subsubsection*{Definition and properties}

Definition \ref{DFN.Wstar_Derivation_QG_HSG} gives noncommutative heat semigroups of quantum Laplacians by extending Definition \ref{DFN.Wstar_Derivation_QG_HSG_L2} via the modified standard pairing. Following Remark \ref{REM.CWstar_Derivation}, this is based on the extension of completely Markovian semigroups in \cite{ART.Cip.1997.NC_Dirichlet_Markov} and uses results in \cite{ART.Cip.1997.NC_Dirichlet_Markov}\cite{ART.Cip_Sav.2003.NC_Dirichlet_Grad}. Proposition \ref{PRP.Wstar_Derivation_QG_HSG_II} and Proposition \ref{PRP.Wstar_Derivation_QG_HSG_Fixed_Part_I}\linebreak collect properties. In particular, note $3)$ in Proposition \ref{PRP.Wstar_Derivation_QG_HSG_Fixed_Part_I} shows sets of states at finite\linebreak distance have identical fixed parts.\par
Let $(\phi,\bpsi,\gamma,\nabla)$ be noncommutative differential structure for tracial AF-$C^{*}$-algebras $(A,\tau)$ and $(B,\omega)$ in $\lc{}f,\theta\rc$-setting.

\begin{dfn}\label{DFN.Wstar_Derivation_QG_HSG_L2}
We define heat semigroup $h:[0,\infty)\longrightarrow\BII\lc{}L^{2}(A,\tau)\rc$ of $\Delta$ by setting

\begin{align}\label{EQ.DFN.Wstar_Derivation_QG_HSG_L2_1}
h_{t}(u):=e^{-t\Delta}(u)
\end{align}

\noindent for all $t\geq 0$ and $u\in L^{2}(A,\tau)$.
\end{dfn}

\begin{ntn}\label{NTN.Wstar_Derivation_QG_HSG_L2}
For all $j\in\mathbb{N}$, let $h^{j}:[0,\infty)\longrightarrow\BII(A_{j})$ denote heat semigroup of $\Delta_{j}$ in Definition \ref{DFN.Wstar_Derivation_QG_HSG_L2} for the induced noncommutative differential structure $\lc\phi_{j},\bpsi_{j},\gamma_{j},\nabla_{\hspace{-0.055cm} j}\rc$.
\end{ntn}

\begin{rem}\label{REM.Wstar_Derivation_QG_HSG_L2}
Note $\Delta\in\UBII\lc{}L^{2}(A,\tau)\rc$ is local by $4)$ in Proposition \ref{PRP.Wstar_Derivation_Compression_II} and $3.1)$ in Proposition \ref{PRP.Wstar_Derivation_QG_I}. Thus Proposition \ref{PRP.Local_Operator} applies, hence $1)$ therein yields orthonormal eigenbasis $\{e_{n}\}_{n\in\mathbb{N}}\subset A_{0}$ of $\Delta$ s.t.~it is furthermore orthonormal eigenbasis of $\pi_{j}^{A}$ for all $j\in\mathbb{N}$. By testing on $A_{0}$ using $4)$ in Proposition \ref{PRP.Wstar_Derivation_Compression_II}, $3.1)$ in Proposition \ref{PRP.Wstar_Derivation_QG_I} shows

\begin{align}\label{EQ.REM.Wstar_Derivation_QG_HSG_L2_1}
\lb\pi_{j}^{A},\pi_{\ker\Delta}^{A}\rb{}=0
\end{align}

\noindent for all $j\in\mathbb{N}$ since $A_{0}\subset L^{2}(A,\tau)$ is $\|.\|_{\tau}$-dense. Alternatively, we derive Equation \ref{EQ.REM.Wstar_Derivation_QG_HSG_L2_1} by calculating on an orthonormal basis as above. Equation \ref{EQ.REM.Wstar_Derivation_QG_HSG_L2_1} thereby generalises to $2.2)$ in Proposition \ref{PRP.Wstar_Derivation_QG_HSG_II}.
\end{rem}

The heat semigroup of $\Delta$ extends as follows. For all $j\in\mathbb{N}$, following Remark \ref{REM.Wstar_Derivation_QG_HSG_L2} note $3.1)$ in Proposition \ref{PRP.Wstar_Derivation_QG_I} lets us apply $1)$ in Proposition \ref{PRP.Wstar_Derivation_Compression_III} in order to get
 
\begin{align}\label{EQ.SSEC.QOT_AC_HSG_5}
h_{t}(x)=\lc{}e^{-t\Delta_{j}}\oplus e^{-t\Delta_{j}^{\perp}}\rc{}(x)=e^{-t\Delta_{j}}(x)=h_{t}^{j}(x)\in A_{j}
\end{align}

\noindent for all $t\geq 0$ and $x\in A_{j}$. For all $j\in\mathbb{N}$, we have symmetric $C^{*}$-derivation $\nabla_{\hspace{-0.055cm} j}:A_{j}\longrightarrow B_{j}$ by $1)$ in Proposition \ref{PRP.Wstar_Derivation_QG_I}. Theorem 8.3 in \cite{ART.Cip_Sav.2003.NC_Dirichlet_Grad} shows we have $C^{*}$-Dirichlet form $u\mapsto \dblv{}\nabla_{\hspace{-0.055cm} j}u\dblv_{\tau}^{2}$ on $A_{j}$ in each case. Using the latter, Theorem 4.11 in \cite{ART.Cip.1997.NC_Dirichlet_Markov} shows we have completely Markovian semigroup $h^{j}:[0,\infty)\longrightarrow\BII(A_{j})$ as well. Note our argument here initially yields Markovianity. Completeness follows by likewise application of both theorems to extensions of symmetric $C^{*}$-derivations to full matrix algebras over finite-dimensional tracial $C^{*}$-algebras. Theorem 2.12 in \cite{ART.Cip.1997.NC_Dirichlet_Markov} shows completely Markovian semigroups and their extensions to Banach dual spaces are given by completely positive dilations.\par


\pagebreak


For all $j\in\mathbb{N}$, we therefore have

\begin{align}\label{EQ.SSEC.QOT_AC_HSG_6}
\dblv{}h_{t}^{j}(x)\dblv_{\infty}\leq \|x\|_{\infty}
\end{align}

\noindent for all $x\in A_{j}$. Using $A_{0}\subset A$ $\|.\|_{\infty}$-dense, Equation \ref{EQ.REM.Wstar_Derivation_QG_HSG_L2_1} and Equation \ref{EQ.SSEC.QOT_AC_HSG_6} then yield extension $h_{t}\in\BII(A)$ of Equation \ref{EQ.DFN.Wstar_Derivation_QG_HSG_L2_1} for all $t\geq 0$. Dualisation of such an extended Equation \ref{EQ.DFN.Wstar_Derivation_QG_HSG_L2_1} defines semigroup $h:[0,\infty)\longrightarrow\BII(A^{*})$ by setting

\begin{align}\label{EQ.SSEC.QOT_AC_HSG_7}
h_{t}(\mu)(x):=e^{-t\Delta}(\mu)(x):=\mu\lc{}h_{t}(x)\rc{}
\end{align}

\noindent for all $t\geq 0$, $\mu\in A^{*}$ and $x\in A$. Following Remark \ref{REM.Wstar_Trace_MSP}, normality moreover restricts Equation \ref{EQ.SSEC.QOT_AC_HSG_7} to

\begin{align}\label{EQ.SSEC.QOT_AC_HSG_8}
\sharp\lc\restr{0.925}{h_{t}}{L^{1}(A,\tau)^{\flat}}\rc\circ\flat\in\BII\lc{}L^{1}(A,\tau)\rc{}    
\end{align}

\noindent for all $t\geq 0$. Equation \ref{EQ.SSEC.QOT_AC_HSG_8} defines semigroup $h:[0,\infty)\longrightarrow\BII\lc{}L^{1}(A,\tau)\rc$ by setting

\begin{align}\label{EQ.SSEC.QOT_AC_HSG_9}
h_{t}(x):=e^{-t\Delta}(x):=\sharp\big(h_{t}(x^{\flat})\big)
\end{align}

\noindent for all $t\geq 0$ and $x\in L^{1}(A,\tau)$. Finally, dualisation of Equation \ref{EQ.SSEC.QOT_AC_HSG_9} and accounting for using the modified standard pairing $L^{\infty}(A,\tau)=L^{1}(A,\tau)^{*}$ as per Equation \ref{EQ.SSEC.QOT_AC_HSG_8} defines semigroup $h:[0,\infty)\longrightarrow\BII\lc{}L^{\infty}(A,\tau)\rc$ by setting

\begin{align}\label{EQ.SSEC.QOT_AC_HSG_10}
h_{t}(x)(y):=e^{-t\Delta}(x)(y):=x^{\flat}\lc{}h_{t}(y)\rc{}
\end{align}

\noindent for all $t\geq 0$, $x\in L^{\infty}(A,\tau)$ and $y\in L^{1}(A,\tau)$. Note Equation \ref{EQ.SSEC.QOT_AC_HSG_10} restricts to extension of Equation \ref{EQ.DFN.Wstar_Derivation_QG_HSG_L2_1} to $A$ for all $x\in A$. Up to musical isomorphisms, all extensions coincide on intersections of domains. Altogether, we have noncommutative heat semigroup of $\Delta$ mapping to $\BII(V)$ if $V=A^{*}$ or $V=L^{p}(A,\tau)$ for $p\in\lset{}1,2,\infty\rset$.

\begin{prp}\label{PRP.Wstar_Derivation_QG_HSG_I}
Let $V=A^{*}$ or $V=L^{p}(A,\tau)$ for $p\in\lset{}1,2,\infty\rset$.

\begin{itemize}
\item[1)] For all $v\in V$, $h_{\infty}(v):=w^{*}$-$\lim_{t\rightarrow\infty}h_{t}(v)$ exists.

\item[2)] For all $t\geq 0$ and $u\in L^{2}(A,\tau)$, we have

\begin{itemize}
\item[2.1)] $h_{\infty}(u)=\pi_{\ker\Delta}^{A}(u)$,

\item[2.2)] $h_{t}(u)\neq 0$ if $u\neq 0$.
\end{itemize}

\begin{reapply}
\end{reapply}

\end{itemize}
\end{prp}
\begin{proof}
Following Remark \ref{REM.Wstar_Trace_MSP}, density of $A_{0}$ and normality imply $\|.\|_{V}$ is determined by testing on $A_{0}$. Let $v\in V$. Equation \ref{EQ.SSEC.QOT_AC_HSG_6} shows $\sup_{t\geq 0}\dblv{}h_{t}(v)\dblv_{V}\leq 4\| v\|_{V}$. Thus $1)$ follows if $\lim_{t\rightarrow\infty}h_{t}(v)(x)$ exists for all $x\in A_{0}$. We require $2.1)$. Following Remark \ref{REM.Wstar_Derivation_QG_HSG_L2} and using Equation \ref{EQ.SSEC.QOT_AC_HSG_5}, we calculate $\pi_{\ker\Delta}^{A}(x)=\|.\|_{\tau}$-$\lim_{t\rightarrow\infty}h_{t}(v)(x)$ for all $x\in A_{0}$ on an orthonormal eigenbasis $\{e_{n}\}_{n\in\mathbb{N}}\subset A_{0}$ of $\Delta$ as per the remark. We obtain $2.1)$ by density. Then $2.1)$ implies $1)$. We directly verify $2.2)$ by likewise calculation. Get $2)$. 
\end{proof}

\begin{dfn}\label{DFN.Wstar_Derivation_QG_HSG}
Let $V=A^{*}$ or $V=L^{p}(A,\tau)$ for $p\in\lset{}1,2,\infty\rset$. We define heat semigroup $h:[0,\infty]\longrightarrow\BII\lc{}L^{2}(A,\tau)\rc$ of $\Delta$ by setting

\begin{align}\label{EQ.DFN.Wstar_Derivation_QG_HSG_1}
h_{t}(v):=e^{-t\Delta}(v)
\end{align}

\noindent for all $t\geq 0$ and $v\in V$.
\end{dfn}

\begin{prp}\label{PRP.Wstar_Derivation_QG_HSG_II}
Let $V=A^{*}$ or $V=L^{p}(A,\tau)$ for $p\in\lset{}1,2,\infty\rset$.

\begin{itemize}
\item[1)] We have strongly continuous semigroup $h:[0,\infty)\longrightarrow\BII(V)$. In particular, we have trace-preserving and completely Markovian semigroup $h:[0,\infty)\longrightarrow\BII\lc{}L^{\infty}(A,\tau)\rc$.

\item[2)] For all $t\in [0,\infty]$, we have

\begin{itemize}
\item[2.1)] $h_{t}$ is positivity-preserving and $w^{*}$-continuous on norm bounded sets,

\item[2.2)] $h_{t}^{j}\circ\resj=h_{t}\circ\resj=\resj\circ\hspace{0.0275cm} h_{t}$ for all $j\in\mathbb{N}$,

\item[2.3)] $\|h_{t}\|_{\BII(V)}\leq 1$ and $h_{t}(1_{A})=1_{A}$,

\item[2.4)] $h_{t}\in\BII\lc{}L^{2}(A,\tau)\rc_{h}$ is local.
\end{itemize}

\begin{reapply}
\end{reapply}

\end{itemize}
\end{prp}
\begin{proof}
By construction, $h:[0,\infty)\longrightarrow\BII(V)$ is a semigroup s.t.~$h_{t}$ is $w^{*}$-continuous on norm bounded sets for all $t\geq 0$. We show $1)$. For all $t\geq 0$, testing for $\|.\|_{V}$ on $A_{0}$ lets us apply Equation \ref{EQ.SSEC.QOT_AC_HSG_6} in order to calculate

\begin{align}\label{EQ.PRP.Wstar_Derivation_QG_HSG_II_1}
\|h_{t}\|_{\BII(V)}\leq 1    
\end{align}

\noindent for all $v\in V$ and $t\geq 0$. Equation \ref{EQ.PRP.Wstar_Derivation_QG_HSG_II_1} implies strong continuity. We extend to $t=\infty$ by letting $t\uparrow\infty$ in the latter equation. Assume $V=L^{\infty}(A,\tau)$. For all $j\in\mathbb{N}$, note $\Delta_{j} 1_{A_{j}}=0$ by the Leibniz rule. Using the latter and $2)$ in Proposition \ref{PRP.AF_Cstar_Unit}, Equation \ref{EQ.SSEC.QOT_AC_HSG_5} lets us calculate $h_{t}(1_{A})=\s$-$\lim_{j\in\mathbb{N}}h_{t}(1_{A_{j}})=\s$-$\lim_{j\in\mathbb{N}}1_{A_{j}}=1_{A}$ for all $t\geq 0$. We extend to $t=\infty$ by letting $t\uparrow\infty$ in our calculation. Moreover, we see $h_{t}\in\BII\lc{}L^{\infty}(A,\tau)\rc$ is trace-preserving for all $t\geq 0$ by testing all $x\in L^{1,\infty}(A,\tau)$ with $y=1_{A}$ as per Equation \ref{EQ.SSEC.QOT_AC_HSG_10}.\par
For all $j\in\mathbb{N}$, our construction ensures $h^{j}:[0,\infty)\longrightarrow\BII(A_{j})$ is completely Markovian. Using $2.2)$ in Proposition \ref{PRP.AF_Cstar_Trace_Dualisation_II}, resp.~$2)$ in Proposition \ref{PRP.AF_Cstar_Unit}, we calculate

\begin{align}\label{EQ.PRP.Wstar_Derivation_QG_HSG_II_2}
h_{t}(x)\otimes I_{n}=w^{*}\textrm{-}\lim_{j\in\mathbb{N}}\hspace{0.025cm} h_{t}(x_{j})\otimes I_{n}\geq 0   
\end{align}

\noindent and

\begin{align}\label{EQ.PRP.Wstar_Derivation_QG_HSG_II_3}
h_{t}(1_{A})\otimes I_{n}=w^{*}\textrm{-}\lim_{j\in\mathbb{N}}\hspace{0.025cm} h_{t}(1_{A_{j}})\otimes I_{n}\leq w^{*}\textrm{-}\lim_{j\in\mathbb{N}}\hspace{0.025cm} 1_{A_{j}}\otimes I_{n}=1_{A}\otimes I_{n}
\end{align}

\noindent for all $n\in\mathbb{N}$ and $x\in L^{\infty}(A,\tau)_{+}$. Equation \ref{EQ.PRP.Wstar_Derivation_QG_HSG_II_2} uses restrictions are positivity-preserving by Proposition \ref{PRP.AF_Cstar_Trace_Dualisation_I}. For all $t\geq 0$, Equation \ref{EQ.PRP.Wstar_Derivation_QG_HSG_II_2} shows $h_{t}$ is completely positive and Equation \ref{EQ.PRP.Wstar_Derivation_QG_HSG_II_3} shows $h_{t}$ is completely Markovian. We are left to show normality in each case. Complete positivity and Proposition \ref{PRP.Wstar_Normal} reduce to $\sigma$-weak continuity. Note the latter is equivalent to $w^{*}$-continuity on norm bounded sets \lc{}cf.~Lemma II.2.5 in \cite{BK.Tak.1979.OpAlg_I} and Proposition \ref{PRP.Wstar_Equivalence}\rc{}. Get $1)$.\par


\pagebreak


Assume the general case. We show $2)$. Since we have $w^{*}$-continuity on norm bounded sets, positivity-preservation and therefore $2.1)$ follows by arguing as for Equation \ref{EQ.PRP.Wstar_Derivation_QG_HSG_II_2} in the general case without tensoring. We know all extensions coincide on intersections of domains. Equation \ref{EQ.SSEC.QOT_AC_HSG_5} shows $2.2)$ and Equation \ref{EQ.SSEC.QOT_AC_HSG_6} shows $2.3)$. Then $2.2)$ implies $2.4)$ at once. Altogether, get $2)$. 
\end{proof}

Definition \ref{DFN.Wstar_Derivation_QG_HSG_Fixed_Part_I} gives fixed parts of positive bounded functionals, and thereby fixed states, under noncommutative heat semigroups of quantum Laplacians. Note states are preserved by $1)$ in Proposition \ref{PRP.Wstar_Derivation_QG_HSG_Fixed_Part_I}, and have identical fixed parts if at finite distance by $3)$ in Proposition \ref{PRP.Wstar_Derivation_QG_HSG_Fixed_Part_I}. Following this, Definition \ref{DFN.Wstar_Derivation_QG_HSG_Fixed_Part_II} gives sets of states which are determined by fixed parts. These help to classify accessibility components.

\begin{dfn}\label{DFN.Wstar_Derivation_QG_HSG_Fixed_Part_I}
For all $\mu\in A^{*}$, $h(\mu):=h_{\infty}(\mu)$ is its fixed part and $h^{\perp}(\mu):=\mu-h(\mu)$ its image part. We call $\xi\in\SII(A)$ a fixed state, or fixed if $h(\xi)=\xi$.
\end{dfn}

\begin{prp}\label{PRP.Wstar_Derivation_QG_HSG_Fixed_Part_I}\hspace{1cm}
\begin{itemize}
\item[1)] For all $\mu\in A_{+}^{*}$, $t\in [0,\infty]$ and $j\in\mathbb{N}$, we have

\begin{itemize}
\item[1.1)] $\dblv{}h_{t}(\mu)\dblv_{A^{*}}=\|\mu\|_{A^{*}}$,

\item[1.2)] $\mu=0$ if $h(\mu)=0$,

\item[1.3)] $\overline{h_{t}(\mu)}_{j}=h_{t}\lc\bar{\mu}_{j}\rc$.
\end{itemize}

\begin{reapply}
\end{reapply}

\item[2)] For all $t\in [0,\infty]$, we have

\begin{itemize}
\item[2.1)] $h_{t}(\SII(A))\subset\SII(A)$,

\item[2.2)] $h_{t}(\mathcal{S}^{\NI}(A))\subset\mathcal{S}^{\NI}(A)$.
\end{itemize}

\begin{reapply}
\end{reapply}

\item[3)] For all $(\mu,w)\in\Admnullone$, we have $h(\mu(0))=h(\mu(1))$. In particular, states at finite distance have identical fixed part.
\end{itemize}
\end{prp}
\begin{proof}
Note $1.1)$ and $1.2)$ follows from $1)$ in Proposition \ref{PRP.AF_Cstar_Trace_Dualisation_II} and trace-preservation as per $1)$ in Proposition \ref{PRP.Wstar_Derivation_QG_HSG_II}. Using $1.1)$ for rescaling as per $1)$ in Definition \ref{DFN.AF_Cstar_Trace_Dualisation_Paths}, get $1.3)$ by $2.2)$ in Proposition \ref{PRP.Wstar_Derivation_QG_HSG_II}. Note Remark \ref{REM.AF_Cstar_Trace_Dualisation_Admissible_Paths}. Equation \ref{EQ.SSEC.QOT_AC_HSG_8} shows normality is preserved under $h_{t}\in\BII(A^{*})$ for all $t\in [0,\infty]$. Then $1)$ implies $2)$. For $3)$, we reduce to the finite-dimensional setting by $2)$ in Corollary \ref{COR.QOT_Distance_AC_I} and $1.3)$.\par
Assume $A$ and $B$ are finite-dimensional. Let $(\mu,w)\in\Admnullone$. Thus the continuity equation and finite-dimensionality imply

\begin{align}\label{EQ.PRP.Wstar_Derivation_QG_HSG_Fixed_Part_I_1}
\sharp\dot{\mu}(t)=\nabla^{*}\pi_{\im\hspace{-0.055cm}\nabla}\lc{}w(t)\rc\in\im\Delta    
\end{align}

\noindent for a.e~$t\in [0,1]$. We moreover have $h\lc\mu(t)\rc{}=\pi_{\ker\Delta}^{A}\lc\mu(t)\rc\in\ker\Delta$ for all $t\in [0,1]$ by $2.1)$ in Proposition \ref{PRP.Wstar_Derivation_QG_HSG_I}. Using the latter, Equation \ref{EQ.PRP.Wstar_Derivation_QG_HSG_Fixed_Part_I_1} implies $3)$ in the finite-dimensional setting. The general case follows as discussed above.
\end{proof}

\begin{dfn}\label{DFN.Wstar_Derivation_QG_HSG_Fixed_Part_II}\hspace{1cm}
\begin{itemize}
\item[1)] For all norm closed convex $K\subset\SII(A)$, set $\Fix_{A}(K):=\big\{\hspace{0.025cm} \mu\in\SII(A)\ \vset\ h(\mu)\in K\hspace{0.025cm} \big\}$.

\item[2)] For all fixed states $\xi\in\SII(A)$, set

\begin{itemize}
\item[2.1)] $\Fix_{A}(\xi):=\Fix\hspace{-0.0325cm} \big(\lset\xi\rset{},A\big)$ and $\Fix_{A}^{\NI}(\xi):=\Fix_{A}(\xi)\cap\mathcal{S}^{\NI}(A)$, \phantom{\big)}

\item[2.2)] $\CII_{A}(\xi):=\big\{\hspace{0.025cm} \mu\in\SII(A)\ \vset\ \mu\sim\xi\hspace{0.025cm} \big\}$ and $\mathcal{C}_{A}^{\NI}(\xi):=\CII_{A}(\xi)\cap\mathcal{S}^{\NI}(A)$. \phantom{\big)}
\end{itemize}

\begin{reapply}
\end{reapply}

\end{itemize}
\end{dfn}

\begin{prp}\label{PRP.Wstar_Derivation_QG_HSG_Fixed_Part_II}
Let $K\subset\SII(A)$ be a norm closed convex subset. If $K\subset\SII(A)$ is a face, then $\Fix_{A}(K)$ is a face. 
\end{prp}
\begin{proof}
Let $\mu\in\Fix_{A}(K)$, $\eta_{0},\eta_{0}\in\SII(A)$ and $t\in (0,1)$ s.t.~$\mu=t\eta_{0}+\lc{}1-t\rc\eta_{1}\in\Fix\lc{}K,A\rc$. We have $h\lc\Fix_{A}(K)\rc\subset K$ and therefore $h(\mu)=th\lc\eta_{0}\rc{}+\lc{}1-t\rc{}h\lc\eta_{1}\rc\in K$. Assume $K$ is a face. Thus $h\lc\eta_{0}\rc{},h\lc\eta_{0}\rc\in K$, hence $\eta_{0},\eta_{1}\in\Fix_{A}(K)$. Norm closedness of $\Fix_{A}(K)$ follows by $2.1)$ in Proposition \ref{PRP.Wstar_Derivation_QG_HSG_II}. Altogether, our claim follows.
\end{proof}


\subsubsection*{Regularisation of normal states under heat flow}

Assuming fixed parts with integrable support, Theorem \ref{THM.Wstar_Derivation_QG_HSG_Regularity} shows heat flow instantaneously regularises normal states to be, possibly unboundedly, invertible up to fixed part. The latter is equivalent to injectivity up to fixed part. Following Remark \ref{REM.NCD_Operator_Compressed_PMO_I}, we know Theorem \ref{THM.NCD_Operator_Compressed_PMO} applies to noncommutative densities in form of Corollary \ref{COR.NCD_Operator_Compressed_PMO_I} given injectivity up to fixed part. Note Remark \ref{REM.Wstar_Derivation_QG_HSG_Regularity_I}. Theorem \ref{THM.Wstar_Derivation_QG_HSG_Regularity} uses Lemma \ref{LEM.Wstar_Derivation_QG_HSG_Regularity_III}. In the finite-dimensional setting, Lemma \ref{LEM.Wstar_Derivation_QG_HSG_Regularity_II} shows Lemma \ref{LEM.Wstar_Derivation_QG_HSG_Regularity_III}, itself obtained from Lemma \ref{LEM.Wstar_Derivation_QG_HSG_Regularity_I}. We show the latter two lemmas by adapting \cite{ART.Sim.1973.Semigroups_Positivity_Preserving} to the AF-$C^{*}$-setting.\par
Let $(\phi,\bpsi,\gamma,\nabla)$ be noncommutative differential structure for tracial AF-$C^{*}$-algebras $(A,\tau)$ and $(B,\omega)$ in $\lc{}f,\theta\rc$-setting.

\begin{lem}\label{LEM.Wstar_Derivation_QG_HSG_Regularity_I}
Let $T\in\BII\lc{}L^{2}(A,\tau)\rc_{h}$ be positivity-preserving. If $T(u)\neq 0$ for all non-zero $u\in L^{2}(A,\tau)_{+}$, then $\lgl u,v\rgl_{\tau}>0$ implies $\lgl T(u),T(v)\rgl_{\tau}>0$ for all $u,v\in L^{2}(A,\tau)_{+}$.
\end{lem}
\begin{proof}
We adapt Lemma 1 in \cite{ART.Sim.1973.Semigroups_Positivity_Preserving}. For this, we require infima in $L^{2}(A,\tau)_{+}$ using partial order generated by positive elements. Definition 4.3 in \cite{ART.Cip.1997.NC_Dirichlet_Markov} gives a wedge operation on $L^{2}(A,\tau)_{h}$ using projections onto closed convex sets of Hilbert spaces. These describe the infima we use as follows. For all $x\in L^{2}(A,\tau)_{h}$, Proposition \ref{PRP.Wstar_NCI_IV} yields $x_{+},x_{-}\in L^{2}(A,\tau)_{+}$ s.t.~$x=x_{+}-x_{-}$, $-x=x_{-}-x_{+}$ and $x_{+}x_{-}=x_{-}x_{+}=0$. Lemma 4.4 in \cite{ART.Cip.1997.NC_Dirichlet_Markov} states

\begin{align}\label{EQ.LEM.Wstar_Derivation_QG_HSG_Regularity_I_1}
\inf\hspace{0.0375cm} \{u,v\}=v-\lc{}u-v\rc_{-}=u\wedge v=v\wedge u=u-\lc{}v-u\rc_{-}=\inf\hspace{0.0375cm} \{v,u\}
\end{align}

\noindent for all $u,v\in L^{2}(A,\tau)_{+}$. If $u,v\in L^{2}(A,\tau)_{+}$ s.t.~$u\wedge v=0$, then Equation \ref{EQ.LEM.Wstar_Derivation_QG_HSG_Regularity_I_1} shows we have $u=\lc{}v-u\rc_{-}$ and $v=\lc{}u-v\rc_{-}=\lc{}v-u\rc_{+}$. For all $u,v\in L^{2}(A,\tau)_{+}$, we use decomposition as per Proposition \ref{PRP.Wstar_NCI_IV} and thereby see $u\wedge v=0$ implies $uv=vu=0$.\par


\pagebreak


We show our claim using the above. Assume $T(u)\neq 0$ for all non-zero $u\in L^{2}(A,\tau)_{+}$. Let $u,v\in L^{2}(A,\tau)_{+}$ s.t.~$\lgl u,v\rgl_{\tau}>0$. Thus traciality and faithfulness imply $uv\neq 0$, hence Equation \ref{EQ.LEM.Wstar_Derivation_QG_HSG_Regularity_I_1} shows $u\wedge v\neq 0$ as discussed above. Note $u,v\geq u\wedge v\geq 0$ by the infimum property. In particular, $u\wedge v\in L^{2}(A,\tau)_{+}$. Ergo $T\lc{}u\wedge v\rc\neq 0$ by hypothesis. We have

\begin{align}\label{EQ.LEM.Wstar_Derivation_QG_HSG_Regularity_I_2}
\lgl T(u),T(v)\rgl_{\tau}=\dblv{}T\lc{}u\wedge v\rc\dblv_{\tau}^{2}+\lgl T\lc{}u-u\wedge v\rc{},T\lc{}u\wedge v\rc\rgl_{\tau}+\lgl T(u),T\lc{}v-u\wedge v\rc\rgl_{\tau}.
\end{align}

\noindent For all $x,y\in L^{2}(A,\tau)_{+}$, we know $\lgl x,y\rgl_{\tau}\geq 0$ by traciality. Positivity-preservation implies the second and third summand in Equation \ref{EQ.LEM.Wstar_Derivation_QG_HSG_Regularity_I_2} are non-negative. Since $T\lc{}u\wedge v\rc\neq 0$ implies $\dblv{}T\lc{}u\wedge v\rc\dblv_{\tau}^{2}>0$, Equation \ref{EQ.LEM.Wstar_Derivation_QG_HSG_Regularity_I_2} shows our claim.
\end{proof}

\begin{lem}\label{LEM.Wstar_Derivation_QG_HSG_Regularity_II}
For all $x\in L^{1,\infty}(A,\tau)_{+}$ and $u\in L^{2}(A,\tau)$, we have

\begin{itemize}
\item[1)] $\lgl xu,u\rgl_{\tau}>0$ implies $\lgl h_{t}(x)u,u\rgl_{\tau}>0$ for all $t\geq 0$,

\item[2)] the map $t\mapsto h_{t}(x,u):=\lgl h_{t}(x)u,u\rgl_{\tau}$ defined on $(0,\infty)$ is either identically zero or has at most finitely many zeros in each open interval $I\subset (0,\infty)$.
\end{itemize}
\end{lem}
\begin{proof}
For all $t\geq 0$, note $2.2)$ in Proposition \ref{PRP.Wstar_Derivation_QG_HSG_I} and $2.1)$ in Proposition \ref{PRP.Wstar_Derivation_QG_HSG_II} imply Lemma \ref{LEM.Wstar_Derivation_QG_HSG_Regularity_I} applies to $T=h_{t}\in\BII\lc{}L^{2}(A,\tau)\rc_{h}$. We show $1)$. Let $x\in L^{1,\infty}(A,\tau)_{+}$ and $u\in L^{2}(A,\tau)$ s.t.~$\lgl xu,u\rgl_{\tau}>0$. Corollary \ref{COR.Wstar_CLRA_IV} reduces the general case to $u\in L^{2,\infty}(A,\tau)$ and Lemma \ref{LEM.Wstar_Derivation_QG_HSG_Regularity_I} shows our claim in this special case.\par
We reduce to $u\in L^{2,\infty}(A,\tau)$. For all $y\in L^{\infty}(A,\tau)_{+}$ and $w\in L^{2}(A,\tau)$, traciality yields

\begin{align}\label{EQ.LEM.Wstar_Derivation_QG_HSG_Regularity_II_1}
y^{\flat}\lc{}ww^{*}\rc{}=\tau\lc{}yww^{*}\rc{}=\lgl yw,w\rgl_{\tau}=\lc{}w^{*}w\rc^{\flat}(y).
\end{align}

\noindent Set $v:=u^{*}u\in L^{1}(A,\tau)_{+}$. For all $n\in\mathbb{N}$, set $v_{n}:=\min\{v,n\}\in L^{1,\infty}(A,\tau)\subset L^{2,\infty}(A,\tau)$. We have $0\leq v_{n}\leq v$ in each case. Using Equation \ref{EQ.LEM.Wstar_Derivation_QG_HSG_Regularity_II_1}, we therefore estimate

\begin{align}\label{EQ.LEM.Wstar_Derivation_QG_HSG_Regularity_II_2}
\lgl h_{t}(x)\sqrt{v_{n}},\sqrt{v_{n}}\rgl_{\tau}=v_{n}^{\flat}\lc{}h_{t}(x)\rc\leq v_{n+1}^{\flat}\lc{}h_{t}(x)\rc\leq v^{\flat}\lc{}h_{t}(x)\rc{}=\lgl h_{t}(x)u,u\rgl_{\tau}
\end{align}

\noindent for all $t\geq 0$ and $n\in\mathbb{N}$. We have $v=\|.\|_{1}$-$\lim_{n\in\mathbb{N}}v_{n}$ \lc{}cf.~$2)$ in Corollary \ref{COR.Wstar_CLRA_IV}\rc{}. Using the latter, Equation \ref{EQ.LEM.Wstar_Derivation_QG_HSG_Regularity_II_2} shows

\begin{align}\label{EQ.LEM.Wstar_Derivation_QG_HSG_Regularity_II_3}
\lgl h_{t}(x)u,u\rgl_{\tau}=\sup_{n\in\mathbb{N}}\hspace{0.025cm} \lgl h_{t}(x)\sqrt{v_{n}},\sqrt{v_{n}}\rgl_{\tau}=\lim_{n\in\mathbb{N}}\hspace{0.025cm} \lgl h_{t}(x)\sqrt{v_{n}},\sqrt{v_{n}}\rgl_{\tau}
\end{align}

\noindent for all $t\geq 0$. Equation \ref{EQ.LEM.Wstar_Derivation_QG_HSG_Regularity_II_3} shows it suffices to consider $u\in L^{2,\infty}(A,\tau)$.\par
We know $x\in L^{2,\infty}(A,\tau)$. Let $u\in L^{2,\infty}(A,\tau)$. We obtain $uu^{*}\in L^{2}(A,\tau)$. Thus $2.1)$ in Proposition \ref{PRP.Wstar_Derivation_QG_HSG_II} implies there exists maximal $\varepsilon\in (0,\infty]$ s.t.~

\begin{align}\label{EQ.LEM.Wstar_Derivation_QG_HSG_Regularity_II_4}
\lgl h_{\frac{t}{2}}(x),h_{\frac{t}{2}}\lc{}uu^{*}\rc\rgl_{\tau}=\tau\lc{}h_{t}(x)uu^{*}\rc{}=\lgl h_{t}(x)u,u\rgl_{\tau}>0    
\end{align}

\noindent for all $t\in [0,\varepsilon)$. If $\varepsilon=\infty$, then our claim follows. If $\varepsilon<\infty$, then Lemma \ref{LEM.Wstar_Derivation_QG_HSG_Regularity_I} shows $\lgl h_{\frac{\varepsilon}{2}}(x),h_{\frac{\varepsilon}{2}}\lc{}uu^{*}\rc\rgl_{\tau}>0$ contradicting maximality. Hence $1)$ holds. The general case follows as discussed above.\par


\pagebreak


We show $2)$. We adapt Lemma 2 in \cite{ART.Sim.1973.Semigroups_Positivity_Preserving}. Let $x\in L^{1,\infty}(A,\tau)_{+}$ and $u\in L^{2}(A,\tau)$. Note $u\in L^{2,\infty}(A,\tau)$ is not required. Following Remark \ref{REM.Wstar_Derivation_QG_HSG_L2}, we have orthonormal eigenbasis $\{e_{n}\}_{n\in\mathbb{N}}\subset A_{0}$ of $\Delta$. For all $n\in\mathbb{N}$, let $\lambda_{n}$ be the eigenvalue of $e_{n}$. Expressing $x=\sum_{n\in\mathbb{N}}\alpha_{n}e_{n}$ and using uniform convergence shows the non-negative map

\begin{align}\label{EQ.LEM.Wstar_Derivation_QG_HSG_Regularity_II_5}
t\mapsto h_{t}(x,u)=\sum_{m\in\mathbb{N}}\vstretch{1.1875}{\Bigg(}\sum_{n\in\mathbb{N}}\frac{(-1)^{m}\cdot \alpha_{n}\lambda_{n}^{m}}{m!}\lgl e_{n}u,u\rgl_{\tau}\vstretch{1.1875}{\Bigg)}\cdot (t-0)^{m}
\end{align}

\noindent is analytic in the right half plane. Using standard arguments for analytic maps \cite{BK.Lan.1999.Complex_Analysis}, we\linebreak see $1)$ implies we either have $h_{t}(x,u)=0$ for all $t\geq 0$ or for at most finitely many $t\in I$ in each open interval $I\subset (0,\infty)$. Get $2)$.
\end{proof}

\begin{lem}\label{LEM.Wstar_Derivation_QG_HSG_Regularity_III}
Let $\xi\in\SII(A)$ be a fixed state and $j\in\mathbb{N}$ s.t.~$\xi_{j}\neq 0$.

\begin{itemize}
\item[1)] We have

\begin{itemize}
\item[1.1)] $\Fix_{A_{j}}^{\NI}\lc\mathcal{F}_{A_{j}}\lc\bar{\xi}_{j}\rc\rc{}=\mathcal{F}_{A_{j}}\lc\bar{\xi}_{j}\rc$,

\item[1.2)] $\bar{\xi}_{j}\in\mathcal{S}_{-1}^{\NI,\infty}\big(A_{j}[\supp\bar{\xi}_{j}]\big)$.
\end{itemize}

\begin{reapply}
\end{reapply}

\item[2)] For all $\mu\in\Fix_{A_{j}}^{\NI}\lc\bar{\xi}_{j}\rc$, we have

\begin{align}\label{EQ.LEM.Wstar_Derivation_QG_HSG_Regularity_III_1}
h_{t}(\mu)\in\mathcal{S}_{-1}^{\NI,\infty}\big(A_{j}[\supp\bar{\xi}_{j}]\big)
\end{align}

\begin{reapply}
\end{reapply}

\noindent for all $t\in (0,\infty]$.
\end{itemize}
\end{lem}
\begin{proof}
Note $1.3)$ in Proposition \ref{PRP.Wstar_Derivation_QG_HSG_Fixed_Part_I} shows $\bar{\xi}_{j}\in\SII(A_{j})$ is a fixed state. Lemma \ref{LEM.Wstar_Derivation_QG_HSG_Regularity_II} and Proposition \ref{PRP.Wstar_Derivation_QG_HSG_Fixed_Part_II} in particular apply to the induced noncommutative differential structure $\lc\phi_{j},\bpsi_{j},\gamma_{j},\nabla_{\hspace{-0.055cm} j}\rc$ using fixed state $\bar{\xi}_{j}\in\SII(A_{j})$. We reduce to the finite-dimensional setting by $1.3)$ in Proposition \ref{PRP.Wstar_Derivation_QG_HSG_Fixed_Part_I}.\par
Assume $A$ and $B$ are finite-dimensional. All states are normal. Lemma \ref{LEM.Support_Projection} shows $\mathcal{F}_{A}(\xi)\subset\SII(A)$ is a face. Thus Proposition \ref{PRP.Wstar_Derivation_QG_HSG_Fixed_Part_II} shows $\Fix_{A}\lc\mathcal{F}_{A}(\xi)\rc\subset\SII(A)$ is one, hence Lemma \ref{LEM.Support_Projection} yields projection $p\in A$ s.t.~

\begin{align}\label{EQ.LEM.Wstar_Derivation_QG_HSG_Regularity_III_2}
\textrm{Fix}_{A}\lc\mathcal{F}_{A}(\xi)\rc{}=\mathcal{S}(A[p]).
\end{align}

\noindent We have $\tau\lc{}p\rc{}<\infty$ as $A_{0}=A\subset\mathfrak{m}_{\tau}$. The semigroup property and Equation \ref{EQ.LEM.Wstar_Derivation_QG_HSG_Regularity_III_2} imply

\begin{align}\label{EQ.LEM.Wstar_Derivation_QG_HSG_Regularity_III_3}
h_{t}\lc\SII(A[p])\rc\subset\SII(A[p])    
\end{align}

\noindent for all $t\in [0,\infty]$. Finite-dimensionality ensures injectivity and invertibility coincide. In particular, get $\mathcal{S}_{>0}^{\NI,\infty}(A[p])=\mathcal{S}_{-1}^{\NI,\infty}(A[p])$. We apply Corollary \ref{COR.Support_Projection_III} accordingly.\par
Note $1)$ in Corollary \ref{COR.Support_Projection_III} states

\begin{align}\label{EQ.LEM.Wstar_Derivation_QG_HSG_Regularity_III_4}
\mathcal{S}_{-1}^{\NI,\infty}(A[p])=\relint\SII(A[p])\subset A[p]_{+}^{*}\cap\GL(A)^{\flat}
\end{align}

\noindent open in norm topology. Equation \ref{EQ.LEM.Wstar_Derivation_QG_HSG_Regularity_III_4} ensures the following equivalence holds. For all $\eta\in\SII(A[p])$, we have $\eta\in\mathcal{S}_{-1}^{\NI,\infty}(A[p])$ if and only if

\begin{align}\label{EQ.LEM.Wstar_Derivation_QG_HSG_Regularity_III_5}
\lgl\eta u,u\rgl_{\tau}\geq\sigma(\eta)\cdot \|u\|_{\tau}^{2}
\end{align}

\noindent for all $u\in A[p]$. Note the below estimate uses strong continuity and trace-preservation as per $1)$ in, as well as positivity-preservation as per $2.1)$ in Proposition \ref{PRP.Wstar_Derivation_QG_HSG_II}. For all $\eta\in\mathcal{S}_{-1}^{\NI,\infty}(A[p])$, Proposition \ref{PRP.Wstar_Derivation_QG_HSG_II}, Equation \ref{EQ.LEM.Wstar_Derivation_QG_HSG_Regularity_III_5} and traciality let us estimate

\begin{align*}
\lgl h(\eta)u,u\rgl_{\tau} = \lim_{t\rightarrow\infty}\hspace{0.025cm} \tau\lc\eta h_{t}\lc{}uu^{*}\rc\rc{} & \geq \sigma(\eta)\cdot \lim_{t\rightarrow\infty}\hspace{0.025cm} \tau\lc{}h_{t}\lc{}uu^{*}\rc\rc{} \phantom{\bigg)} \\
& = \sigma(\eta)\cdot \lim_{t\rightarrow\infty}\hspace{0.025cm} \tau\lc{}uu^{*}\rc{} \phantom{\bigg)} \\
& = \sigma(\eta)\cdot \|u\|_{\tau}^{2} \phantom{\bigg)}
\end{align*}

\noindent for all $u\in A[p]$. Equation \ref{EQ.LEM.Wstar_Derivation_QG_HSG_Regularity_III_3} and the above estimate, either as stated for $t=\infty$ or without taking limits for $t<\infty$, show

\begin{align}\label{EQ.LEM.Wstar_Derivation_QG_HSG_Regularity_III_6}
h_{t}\big(\mathcal{S}_{-1}^{\NI,\infty}(A[p])\big)\subset\mathcal{S}_{-1}^{\NI,\infty}(A[p]) 
\end{align}

\noindent for all $t\in [0,\infty]$. Note $2)$ in Corollary \ref{COR.Support_Projection_III} states we have $\mathcal{F}_{A}(\xi)=\SII(A[p])$ if and only if $\xi\in\mathcal{S}_{-1}^{\NI,\infty}(A[p])$, resp.~$\mathcal{F}_{A}(\xi)\subset\partial\SII(A[p])$ if and only if $\xi\notin\mathcal{S}_{-1}^{\NI,\infty}(A[p])$. If $\xi\in\mathcal{S}_{-1}^{\NI,\infty}(A[p])$ holds, then $\mathcal{F}_{A}(\xi)=\SII(A[p])$ shows $\supp\xi=p$ by $1)$ in Corollary \ref{COR.Support_Projection_I}. Equation \ref{EQ.LEM.Wstar_Derivation_QG_HSG_Regularity_III_2} and Equation \ref{EQ.LEM.Wstar_Derivation_QG_HSG_Regularity_III_4} therefore imply $1)$ in this case.\par
We show $1)$. Assume $\xi\notin\mathcal{S}_{-1}^{\NI,\infty}(A[p])$. Since $\tau\lc{}p\rc{}<\infty$, $\tau\lc{}p\rc^{-1}p^{\flat}\in\mathcal{S}_{-1}^{\NI,\infty}(A[p])$. Note Equation \ref{EQ.LEM.Wstar_Derivation_QG_HSG_Regularity_III_4}. Thus $\partial\SII(A[p])\subset\SII(A[p])$ proper, hence

\begin{align}\label{EQ.LEM.Wstar_Derivation_QG_HSG_Regularity_III_7}
\mathcal{F}_{A}(\xi)\subset\partial\SII(A[p])\subset\SII(A[p])
\end{align}

\noindent proper as well. For all $\eta\in\mathcal{S}_{-1}^{\NI,\infty}(A[p])\neq\emptyset$, Equation \ref{EQ.LEM.Wstar_Derivation_QG_HSG_Regularity_III_2} and Equation \ref{EQ.LEM.Wstar_Derivation_QG_HSG_Regularity_III_7} imply $h(\eta)\in\partial\SII(A[p])$. This contradicts Equation \ref{EQ.LEM.Wstar_Derivation_QG_HSG_Regularity_III_6} for $t=\infty$. Ergo $\xi\in\mathcal{S}_{-1}^{\NI,\infty}(A[p])$. Get $1)$ as discussed above. We show $2)$. Let $\mu\in\Fix_{A}(\xi)$. Using $1.2)$, openness in norm topology as per Equation \ref{EQ.LEM.Wstar_Derivation_QG_HSG_Regularity_III_4} shows there exists $t_{0}\geq 0$ s.t.~

\begin{align}\label{EQ.LEM.Wstar_Derivation_QG_HSG_Regularity_III_8}
h_{t}(\mu)\in\mathcal{S}_{-1}^{\NI,\infty}\big(A[\supp\xi]\big)    
\end{align}

\noindent for all $t\in (t_{0},\infty]$.\par


\pagebreak


Equation \ref{EQ.LEM.Wstar_Derivation_QG_HSG_Regularity_III_9} shows there exists minimal $t_{\mu}\geq 0$ s.t.~Equation \ref{EQ.LEM.Wstar_Derivation_QG_HSG_Regularity_III_8} is satisfied for all $t\in \lc{}t_{\mu},\infty\rb$. Minimality and Equation \ref{EQ.LEM.Wstar_Derivation_QG_HSG_Regularity_III_6} moreover imply

\begin{align}\label{EQ.LEM.Wstar_Derivation_QG_HSG_Regularity_III_9}
h_{t}(\mu)\notin\mathcal{S}_{-1}^{\NI,\infty}\big(A[\supp\xi]\big)    
\end{align}

\noindent for all $t\in \lb{}0,t_{\mu}\rb$. If $t_{\mu}>0$, then finite-dimensionality ensures Equation \ref{EQ.LEM.Wstar_Derivation_QG_HSG_Regularity_III_9} derives a contradiction to $2)$ in Lemma \ref{LEM.Wstar_Derivation_QG_HSG_Regularity_II}. Thus $t_{\mu}=0$ in each case. Get $2)$. The general case follows as discussed above.
\end{proof}

\begin{thm}\label{THM.Wstar_Derivation_QG_HSG_Regularity}
Let $(\phi,\bpsi,\gamma,\nabla)$ be noncommutative differential structure for tracial AF-$C^{*}$-algebras $(A,\tau)$ and $(B,\omega)$ in $\lc{}f,\theta\rc$-setting. Let $\xi\in\SII(A)$ be a fixed state.

\begin{itemize}
\item[1)] Assume $\xi\in\mathcal{S}^{\NI}(A)$ has reducible support.

\begin{itemize}
\item[1.1)] We have

\begin{itemize}
\item[1.1.1)] $\Fix_{A}^{\NI}\lc\mathcal{F}_{A}(\xi)\rc{}=\mathcal{F}_{A}(\xi)$, \phantom{\big)}

\item[1.1.2)] $\supp\xi\in L^{\infty}(A,\tau)_{\nabla}$, $\supp\xi\in\ker\nabla$, and $\nabla$ is $\supp\xi$-compressible. \phantom{\big)}
\end{itemize}

\begin{reapply}
\end{reapply}

\item[1.2)] For all $\mu\in\mathcal{F}_{A}(\xi)$, we have

\begin{align}\label{EQ.THM.Wstar_Derivation_QG_HSG_Regularity_1}
h_{t}(\mu)\in\mathcal{F}_{A}(\xi)    
\end{align}

\begin{reapply}
\end{reapply}

\noindent for all $t\in [0,\infty]$.

\item[1.3)] For all $\mu\in\Fix_{A}^{\NI}(\xi)$ and $j\in\mathbb{N}$ s.t.~$\xi_{j}\neq 0$, we have

\begin{align}\label{EQ.THM.Wstar_Derivation_QG_HSG_Regularity_2}
h_{t}\lc\bar{\mu}_{j}\rc\in\mathcal{S}_{-1}^{\NI,\infty}\big(A_{j}[\supp\bar{\xi}_{j}]\big)
\end{align}

\begin{reapply}
\end{reapply}

\noindent for all $t\in (0,\infty]$.
\end{itemize}

\begin{reapply}
\end{reapply}

\item[2)] Assume $\xi\in\mathcal{S}^{\NI}(A)$ has integrable support.

\begin{itemize}
\item[2.1)] We have 

\begin{itemize}
\item[2.1.1)] $\Fix_{A}^{\NI}\lc\mathcal{F}_{A}(\xi)\rc{}=\mathcal{F}_{A}(\xi)$, \phantom{\big)}

\item[2.1.2)] $\xi\in\mathcal{S}_{>0}^{\NI}\big(A[\supp\xi]\big)$. \phantom{\big)}
\end{itemize}

\begin{reapply}
\end{reapply}

\item[2.2)] For all $\mu\in\Fix_{A}^{\NI}(\xi)$, we have

\begin{align}\label{EQ.THM.Wstar_Derivation_QG_HSG_Regularity_3}
h_{t}(\mu)\in\mathcal{S}_{>0}^{\NI}\big(A[\supp\xi]\big)
\end{align}

\begin{reapply}
\end{reapply}

\noindent for all $t\in (0,\infty]$.
\end{itemize}

\begin{reapply}
\end{reapply}

\end{itemize}
\end{thm}
\begin{proof}
We use $1.2)$ in Proposition \ref{PRP.AF_Cstar_Trace_Dualisation_II} for weak continuity. Furthermore, we use $1.3)$ in Proposition \ref{PRP.Wstar_Derivation_QG_HSG_Fixed_Part_I} to commute restriction and rescaling with application of heat flow. Assume $\xi\in\mathcal{S}^{\NI}(A)$ has reducible support, i.e.~$\supp\xi=\s$-$\lim_{j\in\mathbb{N}}\supp\xi_{j}$.\par
For all $j\in\mathbb{N}$, we see $\bar{\xi}_{j}\in\SII(A_{j})$ is a fixed state if and only if $\xi_{j}\neq 0$. Thus $1.1)$ in Lemma \ref{LEM.Wstar_Derivation_QG_HSG_Regularity_III} implies

\begin{align}\label{EQ.THM.Wstar_Derivation_QG_HSG_Regularity_4}
\textrm{Fix}_{A}^{\NI}\lc\mathcal{F}_{A}(\xi)\rc{}=\lset\mu\in\mathcal{S}^{\NI}(A)\ \vset\ \bar{\mu}_{j}\in\mathcal{F}_{A_{j}}\lc\bar{\xi}_{j}\rc\ \textrm{for a.e.~} j\in\mathbb{N}\rset{}
\end{align}

\noindent by restricting elements on the left-hand side for all $j\in\mathbb{N}$ s.t.~$\xi_{j}\neq 0$, resp.~taking limits of elements on the right-hand side in $w^{*}$-topology. For all $j\in\mathbb{N}$ s.t.~$\xi_{j}\neq 0$, Lemma \ref{LEM.Support_Projection} and $2)$ in Lemma \ref{LEM.Wstar_Derivation_QG_HSG_Regularity_III} show $h_{t}\lc\bar{\mu}_{j}\rc\in\mathcal{F}_{A_{j}}\lc\bar{\xi}_{j}\rc$ and therefore

\begin{align}\label{EQ.THM.Wstar_Derivation_QG_HSG_Regularity_5}
\sharp h_{t}\lc\bar{\mu}_{j}\rc{}=\supp\bar{\xi}_{j}\cdot \sharp h_{t}\lc\bar{\mu}_{j}\rc\cdot \supp\bar{\xi}_{j}.
\end{align}

Equation \ref{EQ.THM.Wstar_Derivation_QG_HSG_Regularity_4} and Equation \ref{EQ.THM.Wstar_Derivation_QG_HSG_Regularity_5} let us calculate

\begin{align}\label{EQ.THM.Wstar_Derivation_QG_HSG_Regularity_6}
\sharp h_{t}(\mu)=w^{*}\textrm{-}\lim_{j\in\mathbb{N}}\hspace{0.025cm} \sharp h_{t}\lc\bar{\mu}_{j}\rc{}=w^{*}\textrm{-}\lim_{j\in\mathbb{N}}\hspace{0.025cm} \supp\bar{\xi}_{j}\cdot \sharp h_{t}\lc\bar{\mu}_{j}\rc\cdot \supp\bar{\xi}_{j}
\end{align}

\noindent for all $\mu\in\Fix_{A}^{\NI}\lc\mathcal{F}_{A}(\xi)\rc$ and $t\in [0,\infty]$. We show the right-hand side of Equation \ref{EQ.THM.Wstar_Derivation_QG_HSG_Regularity_6} is $\supp\xi\cdot \sharp h_{t}(\mu)\cdot \supp\xi$ in each case. For all $x\in L^{\infty}(A,\tau)$, we know $x=\bds$-$\lim_{j\in\mathbb{N}}x_{j}$ by $3)$ in Proposition \ref{PRP.AF_Cstar_Trace_Dualisation_II}. Using weak continuity as for Equation \ref{EQ.THM.Wstar_Derivation_QG_HSG_Regularity_6} and sequential strong continuity of multiplication, Equation \ref{EQ.THM.Wstar_Derivation_QG_HSG_Regularity_5} together with traciality and normality lets us calculate

\begin{align*}
\tau\lc\sharp h_{t}(\mu)x\rc{} = \lim_{j\in\mathbb{N}}\hspace{0.025cm} \tau\lc\sharp h_{t}\lc\bar{\mu}_{j}\rc{}x_{j}\rc{} & = \lim_{j\in\mathbb{N}}\hspace{0.025cm} \tau\lc\sharp h_{t}\lc\bar{\mu}_{j}\rc\cdot \lc\supp\bar{\xi}_{j}\cdot x_{j}\cdot \supp\bar{\xi}_{j}\rc\rc \phantom{\Bigg)} \\
& = \lim_{j\in\mathbb{N}}\hspace{0.025cm} \tau\lc\sharp h_{t}(\mu)\cdot \lc\supp\bar{\xi}_{j}\cdot x_{j}\cdot \supp\bar{\xi}_{j}\rc\rc \phantom{\Bigg)} \\
& = \tau\lc\lc\supp\xi\cdot \sharp h_{t}(\mu)\cdot \supp\xi\rc\cdot x\rc \phantom{\Bigg)}
\end{align*}

\noindent for all $\mu\in\Fix_{A}^{\NI}\lc\mathcal{F}_{A}(\xi)\rc$, $t\in [0,\infty]$ and $x\in L^{\infty}(A,\tau)$. The above calculation at once shows the right-hand side of Equation \ref{EQ.THM.Wstar_Derivation_QG_HSG_Regularity_6} is of claimed form. We therefore have

\begin{align}\label{EQ.THM.Wstar_Derivation_QG_HSG_Regularity_7}
\sharp h_{t}(\mu)=w^{*}\textrm{-}\lim_{j\in\mathbb{N}}\hspace{0.025cm} \supp\bar{\xi}_{j}\cdot \sharp h_{t}\lc\bar{\mu}_{j}\rc\cdot \supp\bar{\xi}_{j}=\supp\xi\cdot \sharp h_{t}(\mu)\cdot \supp\xi
\end{align}

\noindent for all $\mu\in\Fix_{A}^{\NI}\lc\mathcal{F}_{A}(\xi)\rc$ and $t\in [0,\infty]$.\par


\pagebreak


We show $1)$. Equation \ref{EQ.THM.Wstar_Derivation_QG_HSG_Regularity_7} shows $\Fix_{A}^{\NI}\lc\mathcal{F}_{A}(\xi)\rc\subset\mathcal{F}_{A}(\xi)$ by Lemma \ref{LEM.Support_Projection}. We obtain the converse as follows. Using strong continuity as per $1)$ in Proposition \ref{PRP.Wstar_Derivation_QG_HSG_II} to have norm closure, Lemma \ref{LEM.Support_Projection} and Proposition \ref{PRP.Wstar_Derivation_QG_HSG_Fixed_Part_II} yield inclusion of faces and therefore projection $p\in L^{\infty}(A,\tau)$ s.t.~

\begin{align}\label{EQ.THM.Wstar_Derivation_QG_HSG_Regularity_8}
\textrm{Fix}_{A}^{\NI}\lc\mathcal{F}_{A}(\xi)\rc{}=\mathcal{S}^{\NI}(A[p])\subset\mathcal{F}_{A}(\xi)=\mathcal{S}^{\NI}\big(A[\supp\xi]\big)\subset\mathcal{S}^{\NI}(A).
\end{align}

\noindent We have $\xi\in\Fix_{A}^{\NI}\lc\mathcal{F}_{A}(\xi)\rc$. Thus $\supp\xi\leq p$ by Lemma \ref{LEM.Support_Projection}, hence Equation \ref{EQ.THM.Wstar_Derivation_QG_HSG_Regularity_8} shows $\mathcal{F}_{A}(\xi)\subset\Fix_{A}^{\NI}\lc\mathcal{F}_{A}(\xi)\rc$ by $1)$ in Corollary \ref{COR.Support_Projection_I}. Get $1.1.1)$. For all $j\in\mathbb{N}$, note $\Delta 1_{A_{j}}=0$ by the Leibniz rule and $\xi_{j}\in\ker\Delta$ by $2.1)$ in Proposition \ref{PRP.Wstar_Derivation_QG_HSG_I}. Thus $1)$ in Proposition \ref{PRP.Support_Projection_II} implies

\begin{align}\label{EQ.THM.Wstar_Derivation_QG_HSG_Regularity_9}
\supp\xi_{j}\in C^{*}\lc\xi_{j},1_{A_{j}}\rc\subset A_{j}\cap\ker\Delta    
\end{align}

\noindent in each case. Using Corollary \ref{COR.Wstar_Derivation_Projection}, Equation \ref{EQ.THM.Wstar_Derivation_QG_HSG_Regularity_9} and reducible support of $\xi$ shows $1.1.2)$ since $\ker\nabla=\ker\Delta\subset L^{2}(A,\tau)$. Get $1.1)$. Note $1.1.1)$ shows $1.2)$ and $1.3)$ are claims concerning states on $A$ with fixed part $\xi$. Equation \ref{EQ.THM.Wstar_Derivation_QG_HSG_Regularity_4} and Equation \ref{EQ.THM.Wstar_Derivation_QG_HSG_Regularity_7} further reduce to the finite-dimensional setting as per Lemma \ref{LEM.Wstar_Derivation_QG_HSG_Regularity_III}. The latter lemma shows $1.2)$ and $1.3)$ at once. Altogether, get $1)$.\par
We show $2)$. Assume $\xi\in\mathcal{S}^{\NI,\infty}(A)$ has integrable support, i.e.~$\tau\lc\supp\xi\rc{}<\infty$. Ergo Theorem \ref{THM.AF_Support_Projection} shows $\xi$ has reducible support. Thus $1)$ holds, hence $1.1.1)$ implies $2.1.1)$ at once. We further have $2.1.2)$ by $2.1)$ in Corollary \ref{COR.Support_Projection_III} since $\mathcal{F}_{A}(\xi)=\SII\lc{}A[\supp\xi]\rc$ by definition. Get $2.1)$. We reformulate $1.2)$ to

\begin{align}\label{EQ.THM.Wstar_Derivation_QG_HSG_Regularity_10}
h_{t}\bigg(\mathcal{S}^{\NI}\big(A[\supp\xi]\big)\bigg)\subset\mathcal{S}^{\NI}\big(A[\supp\xi]\big)
\end{align}

\noindent for all $t\in [0,\infty]$. Let $\mu\in\mathcal{F}_{A}(\xi)$. For all $t\in [0,\infty]$, Equation \ref{EQ.THM.Wstar_Derivation_QG_HSG_Regularity_10} and Lemma \ref{LEM.Support_Projection} imply $\supp h_{t}(\mu)\leq\supp\xi$. Ergo Theorem \ref{THM.AF_Support_Projection} shows each $h_{t}(\mu)$ has reducible support. Using the latter, $2)$ in Lemma \ref{LEM.Wstar_Derivation_QG_HSG_Regularity_III} shows

\begin{align}\label{EQ.THM.Wstar_Derivation_QG_HSG_Regularity_11}
\supp h_{t}(\mu)=\s\textrm{-}\lim_{j\in\mathbb{N}}\hspace{0.025cm} \supp h_{t}\lc\bar{\mu}_{j}\rc{}=\s\textrm{-}\lim_{j\in\mathbb{N}}\hspace{0.025cm} \supp\bar{\xi}_{j}=\supp\xi    
\end{align}

\noindent for all $t\in (0,\infty]$. Finally, Equation \ref{EQ.THM.Wstar_Derivation_QG_HSG_Regularity_11} shows $2.2)$ by $1)$ in Corollary \ref{COR.Support_Projection_I} and $2.1)$ in Corollary \ref{COR.Support_Projection_III}. Altogether, get $2)$.
\end{proof}

\begin{rem}\label{REM.Wstar_Derivation_QG_HSG_Regularity_I}
We have injectivity of noncommutative densities in general, but do not get smoothing under heat flow as per Equation \ref{EQ.LEM.Wstar_Derivation_QG_HSG_Regularity_III_1}. Injectivity suffices to apply Theorem \ref{THM.NCD_Operator_Compressed_PMO} as per Corollary \ref{COR.NCD_Operator_Compressed_PMO_I}. Coarse graining recovers smoothing under heat flow as per Equation \ref{EQ.THM.Wstar_Derivation_QG_HSG_Regularity_2}. This depends on fixed parts. Such dependence is a uniform condition on accessibility components by $3)$ in Proposition \ref{PRP.Wstar_Derivation_QG_HSG_Fixed_Part_I}.
\end{rem}

We assume integrable support. Theorem \ref{THM.AF_Support_Projection} ensures reducible support. As per Corollary \ref{COR.Wstar_Derivation_Projection} and following Definition \ref{DFN.Wstar_Derivation_QG_Projection}, note it is $1.1.2)$ in Theorem \ref{THM.Wstar_Derivation_QG_HSG_Regularity} which lets us compress quantum gradients with support projections of normal fixed states. We use this throughout our discussion. As per $3)$ in Corollary \ref{COR.Wstar_Derivation_QG_HSG_Regularity}, we moreover combine compressing with such support projections and finite-dimensional approximation. This gives rise to our coarse graining process. Notation \ref{NTN.Wstar_Derivation_QG_HSG_Regularity} fixes conventions. For details on compressing quantum gradients, we refer to Subsection \ref{SSEC.NCDS_NCG_Ubd_Derivation}.

\begin{ntn}\label{NTN.Wstar_Derivation_QG_HSG_Regularity}
Let $\xi\in\mathcal{S}^{\NI}(A)$ be a fixed state with integrable support.

\begin{itemize}
\item[1)] We write $A_{\xi}:=A[\supp\xi]$, $\mathcal{A}_{\xi}:=\mathcal{A}_{L^{\infty}(A_{\xi},\tau)}$ and $L^{\infty}(A_{\xi},\tau)_{\nabla}:=L^{\infty}(A_{\xi},\tau)_{\nabla_{\hspace{-0.055cm} \supp\xi}}$, as well as $L^{2}(B_{\xi},\omega):=\pi_{\supp\xi}\lc{}L^{2}(B,\omega)\rc$. If $A$ and $B$ are finite-dimensional, then we have $A_{\xi}=L^{2}(A_{\xi},\tau)$ and write $B_{\xi}:=L^{2}(B_{\xi},\omega)$. \phantom{\big)}

\item[2)] For all $x\in L^{0}(A_{\xi},\tau)_{+}$, we write $\mathcal{M}_{x,\xi}:=\mathcal{M}_{x,\supp\xi}$ and further $\mathcal{D}_{x,\xi}:=\mathcal{D}_{x^{\flat},x^{\flat}}=\mathcal{D}_{x,\supp\xi}$ if $m_{f}^{-1}\in\SII_{\supp\xi}\lc{}E_{x,x}\rc$. \phantom{\big)}

\item[3)] We write $\nabla_{\hspace{-0.055cm} \xi}:=\nabla_{\hspace{-0.055cm} \supp\xi}=\nabla_{L^{\infty}(A_{\xi},\tau)}$ and $\Delta_{\xi}:=\Delta_{\supp\xi}=\Delta_{L^{\infty}(A_{\xi},\tau)}$. \phantom{\big)}
\end{itemize}
\end{ntn}

\begin{cor}\label{COR.Wstar_Derivation_QG_HSG_Regularity}
Let $\xi\in\mathcal{S}^{\NI}(A)$ be a fixed state with integrable support.

\begin{itemize}
\item[1)] We have

\begin{itemize}
\item[1.1)] $\supp\xi$-compressed symmetric $W^{*}$-derivation $\nabla_{\hspace{-0.055cm} \xi}:\AII_{\xi}\longrightarrow L^{2}(B_{\xi},\omega)$, \phantom{\big)}

\item[1.2)] $\supp\xi$-compressed Laplacian $\Delta_{\xi}=\mathrlap{\phantom{\nabla}_{\hspace{-0.055cm} \xi}}\nabla^{*}\nabla_{\hspace{-0.055cm} \xi}$. \phantom{\big)}
\end{itemize}

\begin{reapply}
\end{reapply}

\item[2)] For all $t\geq 0$ and $h_{t}\in\BII\lc{}L^{2}(A,\tau)\rc$, we have

\begin{itemize}
\item[2.1)] $\lb{}h_{t},\pi_{\supp\xi}\rb{}=0$, \phantom{\big)}

\item[2.2)] $\com_{L^{2}(A_{\xi},\tau)}h_{t}=e^{-t\Delta_{\xi}}$. \phantom{\big)}
\end{itemize}

\begin{reapply}
\end{reapply}

\item[3)] We have $L^{\infty}(A_{\xi},\tau)_{\nabla}\subset\dom\nabla$ and 

\begin{itemize}
\item[3.1)] $\pi_{\supp\xi}(u)=\|.\|_{\nabla}\textrm{-}\lim_{j\in\mathbb{N}}\pi_{\supp\xi_{j}}(u_{j})$ for all $u\in\dom\nabla$, \phantom{\big)}

\item[3.2)] $x=\bds\hspace{-0.075cm}^{\nabla}\textrm{-}\lim_{j\in\mathbb{N}}\pi_{\supp\xi_{j}}(x_{j})$ for all $x\in L^{\infty}(A_{\xi},\tau)_{\nabla}$. \phantom{\big)}
\end{itemize}

\begin{reapply}
\end{reapply}

\item[4)] We have $\dom\nabla\cap L^{2}(A_{\xi},\tau)\subset\dom\nabla_{\hspace{-0.055cm} \xi}$ and

\begin{itemize}
\item[4.1)] $\dom\nabla_{\hspace{-0.055cm} \xi}=\big\{\hspace{0.025cm} u\in L^{2}(A_{\xi},\tau)\ \vset\ u=\|.\|_{\nabla}\textrm{-}\lim_{j\in\mathbb{N}}\pi_{\supp\xi_{j}}(u_{j})\hspace{0.025cm} \big\}$, \phantom{\big)}

\item[4.2)] $L^{\infty}(A_{\xi},\tau)_{\nabla}=\big\{\hspace{0.025cm} x\in L^{\infty}(A_{\xi},\tau)\ \vset\ x=\bds^{\nabla}\textrm{-}\lim_{j\in\mathbb{N}}\pi_{\supp\xi_{j}}(x_{j})\hspace{0.025cm} \big\}$. \phantom{\big)}
\end{itemize}

\begin{reapply}
\end{reapply}

\end{itemize}
\end{cor}
\begin{proof}
We see $1)$ in Corollary \ref{COR.AF_Cstar_Bimodule_Projection} implies $\AII_{\xi}=\supp\xi\cdot A_{0}\cdot \supp\xi\subset A$ using algebra multiplication as per Definition \ref{DFN.Wstar_Derivation_Compression_II}, resp.~$L^{2}(B_{\xi},\omega)=\supp\xi\cdot L^{2}(B,\omega)\cdot \supp\xi$ using AF-$C^{*}$-bimodule action. We know $\nabla$ is $\supp\xi$-compressible by $1.1.2)$ in Theorem \ref{THM.Wstar_Derivation_QG_HSG_Regularity}. We have $1)$ by Corollary \ref{COR.Wstar_Derivation_Projection} and $2)$ in Proposition \ref{PRP.Wstar_Derivation_QG_II}. By testing on $A_{0}$ using $4)$ in Proposition \ref{PRP.Wstar_Derivation_Compression_II}, $1)$ implies $2.1)$ since $A_{0}\subset L^{2}(A,\tau)$ is $\|.\|_{\tau}$-dense. Moreover, $1)$ implies $2.2)$ by Corollary \ref{PRP.Wstar_Derivation_Compression_III}. Get $2)$.\par


\pagebreak


We show $3)$. The latter implies $4)$ immediately. Using sequential strong continuity of multiplication, we readily see reducible support implies $3.1)$ since $u=\|.\|_{\nabla}$-$\lim_{j\in\mathbb{N}}u_{j}$ for all $u\in\dom\nabla$ by $4.1)$ in Proposition \ref{PRP.Wstar_Derivation_QG_I}. We likewise obtain $3.2)$ if $x=\bds^{\nabla}$-$\lim_{j\in\mathbb{N}}x_{j}$ for all $x\in L^{\infty}(A_{\xi},\tau)_{\nabla}$. Arguing as for Equation \ref{EQ.THM.Wstar_Derivation_QG_HSG_Regularity_7}, we use $3)$ in Proposition \ref{PRP.AF_Cstar_Trace_Dualisation_II} and reducible support to calculate

\begin{align}\label{EQ.COR.Wstar_Derivation_QG_HSG_Regularity_1}
x=\s\textrm{-}\lim_{j\in\mathbb{N}}\hspace{0.025cm} \supp\xi_{j}\cdot x_{j}\cdot \supp\xi_{j}=\supp\xi\cdot x\cdot \supp\xi
\end{align}

\noindent for all $x\in L^{\infty}(A_{\xi},\tau)_{\nabla}$. We further have $\supp\xi\in L^{2}(A,\tau)$ and therefore

\begin{align}\label{EQ.COR.Wstar_Derivation_QG_HSG_Regularity_2}
L^{\infty}(A_{\xi},\tau)=\pi_{\supp\xi}\lc{}L^{\infty}(A,\tau)\rc\subset L^{2}(A_{\xi},\tau)    
\end{align}

\noindent by integrable support \lc{}cf.~Proposition \ref{PRP.Wstar_L2Red_III}\rc{}. Equation \ref{EQ.COR.Wstar_Derivation_QG_HSG_Regularity_1} and Equation \ref{EQ.COR.Wstar_Derivation_QG_HSG_Regularity_2} show $L^{\infty}(A_{\xi},\tau)_{\nabla}\subset\dom\nabla$. Using sequential strong continuity of multiplication and $3.1)$, note reducible support implies $3.2)$ by $3)$ in Proposition \ref{PRP.AF_Cstar_Trace_Dualisation_II}. Thus $3)$, hence $4)$ holds.
\end{proof}

\begin{rem}\label{REM.Wstar_Derivation_QG_HSG_Regularity_II}
Following $2)$ in Corollary \ref{COR.Wstar_Derivation_QG_HSG_Regularity}, the noncommutative heat semigroup of $\nabla_{\hspace{-0.055cm} \xi}$ considered as symmetric $C^{*}$-derivation is given by

\begin{align}\label{EQ.REM.Wstar_Derivation_QG_HSG_Regularity_II_1}
t\mapsto\textrm{com}_{L^{2}(A_{\xi},\tau)}\hspace{0.05cm} h_{t}=e^{-t\Delta_{\xi}}\in\BII\lc{}L^{2}(A_{\xi},\tau)\rc{}.
\end{align}

\noindent Since we only consider semigroups as above if we compress with support projections of normal fixed states, we do not distinguish any from $h:[0,\infty)\longrightarrow\BII\lc{}L^{2}(A,\tau)\rc$.
\end{rem}


\subsection{Classifying normal accessibility components}\label{SSEC.QOT_AC_RM}

Assuming spectral gaps of quantum Laplacians and fixed parts, Theorem \ref{THM.QOT_Distance_AC_L2} classifies accessibility components of square integrable normal states using fixed parts by showing each one is a norm closed convex subsets of all such states with identical fixed part. Spectral gaps ensure such fixed parts themselves are square integrable normal states with integrable support. In the finite-dimensional setting, assumptions as above are satisfied and we classify all accessibility components using fixed parts.\par
In the finite-dimensional setting, relative interiors are embedded submanifolds, as well as connected Riemannian manifolds with Riemannian metric induced by the given quasi-entropy. Theorem \ref{THM.RM} shows each in turn induces the given quantum optimal transport distance, and Theorem \ref{THM.QOT_Distance_AC_L2} ensures their norm closures are accessibility components. Theorem \ref{THM.Wstar_Derivation_QG_HSG_Regularity} therefore links the finite-dimensional Riemannian case to the general one by compression, finite-dimensional approximation and heat flow. In Chapter \ref{CH.L2W}, we commonly reduce to the finite-dimensional Riemannian setting. This is a fundamental reason to require, from a purely technical point of view, compatibility with compression and finite-dimensional approximation. Standard reference for differential and Riemannian geometry is \cite{BK.Lan.1995.Riemannian_Manifolds}.


\subsubsection*{Embedded submanifolds of states in the finite-dimensional setting}

We prepare our discussion further below. Let $(\phi,\bpsi,\gamma,\nabla)$ be noncommutative differential structure for tracial AF-$C^{*}$-algebras $(A,\tau)$ and $(B,\omega)$ in $\lc{}f,\theta\rc$-setting. Assume $A$ and $B$ are finite-dimensional.

\begin{prp}\label{PRP.RM_Compressed_PMO}
Let $p\in A$ be a projection. For all $x,y>0$ in $A[p]$ and $u\in B[p]$, we have 

\begin{itemize}
\item[1)] $\mathcal{I}^{f,\theta}\lc{}x^{\flat},y^{\flat},u^{\flat}\rc{}=\lgl\mathcal{D}_{x,y,p}^{\theta}(u),u\rgl_{\omega}$,

\item[2)] $0<\sigma(x)^{\frac{\theta}{2}}\sigma(y)^{\frac{\theta}{2}}\cdot \|u\|_{\omega}^{2}\leq \lgl\mathcal{M}_{x,y,p}^{\theta}(u),u\rgl_{\omega}$.
\end{itemize}
\end{prp}
\begin{proof}
Following Remark \ref{REM.Quadratic_Form}, we have $1)$ by $3)$ in Corollary \ref{COR.NCD_Operator_Compressed_PMO_I}. We show $2)$. The geometric operator mean is the minimal symmetric one \lc{}cf.~Theorem 4.5 in \cite{ART.And_Kub.1979.Operator_Means}\rc{}. Since $x,y>0$ in $A[p]$, evaluating the geometric operator mean in $L_{x,p}$ and $R_{y,p}$ yields

\begin{align}\label{EQ.PRP.RM_Compressed_PMO_1}
0<\sigma\lc{}L_{x,p}\rc^{\frac{\theta}{2}}\sigma\lc{}R_{y,p}\rc^{\frac{\theta}{2}}\cdot \|u\|_{\omega}^{2}\leq \lgl\mathcal{M}_{x,y,p}^{\theta}(u),u\rgl_{\omega}    
\end{align}

\noindent for all $u\in B[p]$. Equation \ref{EQ.PRP.RM_Compressed_PMO_1} shows $2)$ by Proposition \ref{PRP.AF_Support_Projection}.
\end{proof}

Let $\xi\in\mathcal{S}(A_{\xi})$ be a fixed state. We use Notation \ref{NTN.PO}. Restricting the GNS-inner product of $\tau$ yields real Hilbert space inner product of $A_{\xi,h}=A_{\xi}\cap A_{h}$.

\begin{prp}\label{PRP.RM_Embedded_Submanifold_I}
Let $\xi\in\SII(A)$ be a fixed state.

\begin{itemize}
\item[1)] We have $\Delta_{\xi}\in\BII(A_{\xi})_{h}$, $\supp\xi\in\ker\Delta_{\xi}$ and $\im\Delta_{\xi}=\im\Delta\cap A_{\xi}$.

\item[2)] Setting $I(\Delta_{\xi}):=\im\Delta_{\xi}\cap A_{\xi,h}$ and $K(\Delta_{\xi}):=\langle\supp\xi\rangle_{\mathbb{R}}^{\perp}\subset\ker\Delta_{\xi}\cap A_{\xi,h}$ yields orthogonal decomposition

\begin{align}\label{EQ.PRP.RM_Embedded_Submanifold_I_1}
A_{\xi,h}=I(\Delta_{\xi})\oplus\langle\supp\xi\rangle_{\mathbb{R}}\oplus K(\Delta_{\xi}).
\end{align}

\begin{reapply}
\end{reapply}

\end{itemize}
\end{prp}
\begin{proof}
We known $1)$ by $1.1.2)$ in Theorem \ref{THM.Wstar_Derivation_QG_HSG_Regularity} and $1.2)$ in Corollary \ref{COR.Wstar_Derivation_QG_HSG_Regularity}. We have $\Delta\lc{}A_{h}\rc\subset A_{h}$ by symmetry of $\nabla$. Thus $1)$ implies $2)$ at once.
\end{proof}

We have real Hilbert space projections 

\begin{align}\label{EQ.SSEC.QOT_AC_RM_1}
\pi_{I(\Delta_{\xi})}^{A}:A_{\xi,h}\longrightarrow I(\Delta_{\xi}),\ \pi_{K(\Delta_{\xi})}^{A}:A_{\xi,h}\longrightarrow K(\Delta_{\xi}).   
\end{align}

\noindent We know $I(\Delta_{\xi}),K(\Delta_{\xi})\subset\ker\tau$. Furthermore, we know $\tau\lc\supp\xi\rc{}>0$ by faithfulness and have $\dim_{\mathbb{R}}\im_{\mathbb{R}}\tau\vert_{A_{\xi,h}^{*}}=1$. For all $\mu\in A_{\xi,h}^{*}$, Equation \ref{EQ.PRP.RM_Embedded_Submanifold_I_1} yields decomposition

\begin{align}\label{EQ.SSEC.QOT_AC_RM_2}
\mu=\pi_{I(\Delta_{\xi})}^{A}\lc\sharp\mu\rc^{\flat}+\|\mu\|_{A^{*}}\cdot \tau\lc\supp\xi\rc^{-1}\supp\xi^{\flat}+\pi_{K(\Delta_{\xi})}^{A}\lc\sharp\mu\rc^{\flat}. 
\end{align}

\begin{dfn}\label{DFN.RM_Embedded_Submanifold}
Let $\xi\in\SII(A)$ be a fixed state.

\begin{itemize}
\item[1)] We define $\mathfrak{P}_{\xi}:A_{\xi}^{*}\longrightarrow K(\Delta_{\xi})^{\flat}$ by setting

\begin{align}\label{EQ.DFN.RM_Embedded_Submanifold_1}
\mathfrak{P}_{\xi}(\mu):=\pi_{K(\Delta_{\xi})}^{A}\lc\sharp\mu\rc^{\flat}
\end{align}

\begin{reapply}
\end{reapply}

\noindent for all $\mu\in A_{\xi}^{*}$.

\item[2)] Set $\vartheta(\xi):=\mathfrak{P}_{\xi\vert\mathcal{S}_{-1}^{\NI,\infty}(A_{\xi})}^{-1}\lc\pi_{K(\Delta_{\xi})}^{A}\big(\sharp\xi\big)^{\flat}\rc$.
\end{itemize}
\end{dfn}

\begin{ntn}
Let $X$ be a smooth manifold. We write $TX$ for its tangent bundle. We further write $T_{\mu}X$ for the tangent space upon evaluation at $\mu\in X$.
\end{ntn}

\begin{prp}\label{PRP.RM_Embedded_Submanifold_II}
Let $\xi\in\SII(A)$ be a fixed state. We have

\begin{itemize}
\item[1)] embedded submanifold

\begin{align}\label{EQ.PRP.RM_Embedded_Submanifold_II_1}
\vartheta(\xi)=\relint\Fix\hspace{0cm}_{A}^{\NI}(\xi)\subset\mathcal{S}_{-1}^{\NI,\infty}(A_{\xi}),
\end{align}

\begin{reapply}
\end{reapply}

\item[2)] trivial tangent bundle $T\vartheta(\xi)=\vartheta(\xi)\times I(\Delta_{\xi})^{\flat}$.
\end{itemize}
\end{prp}
\begin{proof}
Using $2.1)$ in Proposition \ref{PRP.Wstar_Derivation_QG_HSG_I}, Equation \ref{EQ.SSEC.QOT_AC_RM_2} shows

\begin{align}\label{EQ.PRP.RM_Embedded_Submanifold_II_2}
h(\mu)^{\perp}=\pi_{I(\Delta_{\xi})}^{A}\lc\sharp h^{\perp}(\mu)\rc^{\flat},\ h(\mu)=\tau\lc\supp\xi\rc^{-1}\supp\xi^{\flat}+\pi_{K(\Delta_{\xi})}^{A}\lc\sharp h(\mu)\rc^{\flat}
\end{align}

\noindent for all $\mu\in\SII(A_{\xi})$. Equation \ref{EQ.PRP.RM_Embedded_Submanifold_II_2} implies

\begin{align}\label{EQ.PRP.RM_Embedded_Submanifold_II_3}
\textrm{Fix}_{A}(\xi)=\mathfrak{P}_{\xi\vert\SII(A_{\xi})}^{-1}\lc\pi_{K(\Delta_{\xi})}^{A}\big(\sharp\xi\big)^{\flat}\rc{}.
\end{align}

\noindent Arguing as for $1)$ in Corollary \ref{COR.Support_Projection_III} but using $\xi\in\mathcal{S}_{-1}^{\NI,\infty}(A_{\xi})$ in Equation \ref{EQ.COR.Support_Projection_III_2} rather than rescaled $\supp\xi$ under the flat operator, we directly verify

\begin{align}\label{EQ.PRP.RM_Embedded_Submanifold_II_4}
\relint\textrm{Fix}_{A}(\xi)=\textrm{Fix}_{A}(\xi)\cap\mathcal{S}_{-1}^{\NI,\infty}(A_{\xi}).
\end{align}

\noindent Equation \ref{EQ.PRP.RM_Embedded_Submanifold_II_3} and Equation \ref{EQ.PRP.RM_Embedded_Submanifold_II_4} show Equation \ref{EQ.PRP.RM_Embedded_Submanifold_II_1}. Thus Equation \ref{EQ.PRP.RM_Embedded_Submanifold_II_2} shows smooth paths in $\mathcal{S}_{-1}^{\NI,\infty}(A_{\xi})$ with image in $\vartheta(\xi)$ vary in $I(\Delta_{\xi})^{\flat}$ only, hence Equation \ref{EQ.PRP.RM_Embedded_Submanifold_II_1} implies $1)$ and therefore $2)$ by the submersion theorem \cite{BK.Lan.1995.Riemannian_Manifolds}.
\end{proof}


\subsubsection*{Riemannian metrics induced by quasi-entropies}

Using bounded operators in Definition \ref{DFN.RM_I} determined by quasi-entropies, Definition \ref{DFN.RM_II} gives Riemannian metrics on embedded submanifolds as per Proposition \ref{PRP.RM_Embedded_Submanifold_II}. Restricted to each such embedded submanifold, Theorem \ref{THM.RM} shows the Riemannian distance is the quantum optimal transport distance given by the quasi-entropy inducing Riemannian metric.\par
Let $(\phi,\bpsi,\gamma,\nabla)$ be noncommutative differential structure for tracial AF-$C^{*}$-algebras $(A,\tau)$ and $(B,\omega)$ in $\lc{}f,\theta\rc$-setting. Assume $A$ and $B$ are finite-dimensional. Let $\xi\in\SII(A)$ be a fixed state. This is the finite-dimensional Riemannian setting. We use the following throughout our discussion. Get $\Delta\vert_{\im\Delta}>0$ in $\BII\lc\im\Delta\rc$ by finite-dimensionality. Note $1)$ in Corollary \ref{COR.Wstar_Derivation_QG_HSG_Regularity} shows $\nabla(A_{\xi})\subset B_{\xi}$ and $\nabla^{*}(B_{\xi})\subset A_{\xi}$ by $\supp\xi$-compressibility. For all $x\in A_{\xi,+}$, we see $1)$ in Lemma \ref{LEM.NCD_Operator_Compressed_PMO_I} implies

\begin{align}\label{EQ.SSEC.QOT_AC_RM_3}
\mathcal{M}_{x,\xi}^{\theta}=\restr{0.925}{\mathcal{M}_{x}^{\theta}}{B_{\xi}}.    
\end{align}

\noindent Equation \ref{EQ.SSEC.QOT_AC_RM_3} lets us suppress, upon restriction to $B_{\xi}$, compressing with $\supp\xi$. We suppress accordingly in Definition \ref{DFN.RM_I}.

\begin{dfn}\label{DFN.RM_I}
For all $\mu\in\vartheta(\xi)$, set

\begin{itemize}
\item[1)] $\mathfrak{F}_{\mu}:=\nabla^{*}\mathcal{M}_{\sharp\mu}^{\theta}\nabla\in\BII(\im\Delta_{\xi},A_{\xi})$,

\item[2)] $\mathfrak{G}_{\mu}:=\mathcal{M}_{\sharp\mu}^{\theta}\nabla\in\BII(\im\Delta_{\xi},B_{\xi})$.
\end{itemize}
\end{dfn}

For all $\mu\in\vartheta(\xi)$, we have $\sharp\mu>0$ in $A_{\xi}$ and therefore $\mathfrak{F}_{\mu},\mathrlap{\phantom{\mathfrak{F}}_{\mu}}\mathfrak{F}^{-1}>0$ in $\BII(\im\Delta_{\xi})$ by $1)$ in Proposition \ref{PRP.RM_Embedded_Submanifold_II}. Note $\nabla^{*}\mathfrak{G}_{\mu}=\mathfrak{F}_{\mu}$ in each case by definition.

\begin{prp}\label{PRP.RM_I}
For all $\mu\in\vartheta(\xi)$, we have

\begin{itemize}
\item[1)] $\mathfrak{F}_{\mu},\mathrlap{\phantom{\mathfrak{F}}_{\mu}}\mathfrak{F}^{-1}>0$ in $\BII(\im\Delta_{\xi})$ and $\big\|\mathrlap{\phantom{\mathfrak{F}}_{\mu}}\mathfrak{F}^{-1}\big\|_{\BII(\im\Delta_{\xi})}\leq\sigma(\Delta)^{-1}\sigma(\mu)^{-\theta}$, \phantom{\vstretch{0.75}{\bigg)}}

\item[2)] $\mathfrak{F}_{\mu}\lc{}I(\Delta_{\xi})\rc\subset I(\Delta_{\xi})$ and $\mathrlap{\phantom{\mathfrak{F}}_{\mu}}\mathfrak{F}^{-1}\lc{}I(\Delta_{\xi})\rc\subset I(\Delta_{\xi})$, \phantom{\vstretch{0.75}{\bigg)}}

\item[3)] $\nabla^{*}\mathfrak{G}_{\mu}\mathrlap{\phantom{\mathfrak{F}}_{\mu}}\mathfrak{F}^{-1}=\id_{\im\Delta_{\xi}}$. \phantom{\vstretch{0.75}{\bigg)}}
\end{itemize} 
\end{prp}
\begin{proof}
Let $\mu\in\vartheta(\xi)$. Get $1)$ by $1)$ in Proposition \ref{PRP.RM_Embedded_Submanifold_II}, resp.~$2)$ in Proposition \ref{PRP.RM_Compressed_PMO}. If $\mathfrak{F}_{\mu}\lc{}I(\Delta_{\xi})\rc\subset I(\Delta_{\xi})$, then $\mathrlap{\phantom{\mathfrak{F}}_{\mu}}\mathfrak{F}^{-1}\lc{}I(\Delta_{\xi})\rc\subset I(\Delta_{\xi})$. Note $\nabla$ and $\nabla^{*}$ intertwine adjoining and $\gamma$ by symmetry, resp.~$5)$ in Proposition \ref{PRP.Wstar_Derivation_QG_I}. Symmetry of $f$ implies $\mathcal{M}_{\sharp\mu}^{\theta}\circ\gamma=\gamma\circ \mathcal{M}_{\sharp\mu}^{\theta}$ by $1)$ Corollary \ref{COR.AF_NCD_FC_III}. Get $2)$. We have $3)$ by definition.
\end{proof}

\begin{dfn}\label{DFN.RM_II}
For all $\mu\in\vartheta(\xi)$, set

\begin{align}\label{EQ.DFN.RM_II_1}
g_{\mu}^{\xi}(u,v):=\lgl\mathrlap{\phantom{\mathfrak{F}}_{\mu}}\mathfrak{F}^{-1}\lc\sharp u\rc{},\sharp v\rgl_{\tau}    
\end{align}

\noindent for all $u,v\in I(\Delta_{\xi})^{\flat}$.
\end{dfn}

\begin{prp}\label{PRP.RM_II}\hspace{1cm}
\begin{itemize}
\item[1)] We have connected Riemannian manifold $(\vartheta(\xi),g^{\xi})$.

\item[2)] For all $\mu\in\vartheta(\xi)$, we have

\begin{align}\label{EQ.PRP.RM_II_1}
\mathcal{I}^{f,\theta}\lc\mu,\mu,\lc\mathfrak{G}_{\mu}\mathrlap{\phantom{\mathfrak{F}}_{\mu}}\mathfrak{F}^{-1}\lc\sharp u\rc\rc^{\flat}\rc{}=g_{\mu}^{\xi}(u,u)\leq\sigma(\Delta)^{-1}\sigma(\mu)^{-\theta}\cdot \|\sharp u\|_{\tau}^{2}
\end{align}

\begin{reapply}
\end{reapply}

\noindent for all $u\in I(\Delta_{\xi})^{\flat}$.

\item[3)] Let $\mu\in\vartheta(\xi)$, $u\in I(\Delta_{\xi})^{\flat}$ and $w\in B^{*}$. If $\sharp u=\nabla^{*}\sharp w$, then

\begin{align}\label{EQ.PRP.RM_II_2}
\mathcal{I}^{f,\theta}\lc\mu,\mu,\lc\mathfrak{G}_{\mu}\mathrlap{\phantom{\mathfrak{F}}_{\mu}}\mathfrak{F}^{-1}\lc\sharp u\rc\rc^{\flat}\rc\leq\mathcal{I}^{f,\theta}(\mu,\mu,w).
\end{align}

\begin{reapply}
\end{reapply}

\noindent Furthermore, we have equality in Equation \ref{EQ.PRP.RM_II_2} if and only if $\sharp w=\mathfrak{G}_{\mu}\mathrlap{\phantom{\mathfrak{F}}_{\mu}}\mathfrak{F}^{-1}\lc\sharp u\rc$.
\end{itemize}
\end{prp}
\begin{proof}
The map $\mu\mapsto\mathfrak{F}_{\mu}$ from $\vartheta(\xi)$ to $\BII(\im\Delta_{\xi})_{>0}\subset\GL\lc\BII\lc\im\Delta\rc\rc$ is smooth and invertible by $1)$ and $2)$ in Proposition \ref{PRP.RM_I}. Get $1)$. The identity in Equation \ref{EQ.PRP.RM_II_1} follows by $1)$ in Proposition \ref{PRP.RM_Compressed_PMO}, its subsequent estimate by $1)$ in Proposition \ref{PRP.RM_I}. Get $2)$.\par
We show $3)$. Let $\mu\in\vartheta(\xi)$, $u=x^{\flat}\in I(\Delta_{\xi})^{\flat}$ and $w\in B^{*}$. Assume $x=\nabla^{*}\sharp w$. Then $\nabla^{*}\sharp w=\nabla^{*}\mathfrak{G}_{\mu}\mathrlap{\phantom{\mathfrak{F}}_{\mu}}\mathfrak{F}^{-1}(x)$ by $3)$ in Proposition \ref{PRP.RM_I}. Set $y:=\sharp w-\mathfrak{G}_{\mu}\mathrlap{\phantom{\mathfrak{F}}_{\mu}}\mathfrak{F}^{-1}(x)\in\ker\nabla^{*}$. Using $2)$, get

\begin{align}\label{EQ.PRP.RM_II_3}
\mathcal{I}^{f,\theta}\lc\mu,\mu,\lc\mathfrak{G}_{\mu}\mathrlap{\phantom{\mathfrak{F}}_{\mu}}\mathfrak{F}^{-1}(x)\rc^{\flat}+y^{\flat}\rc{}=g_{\mu}^{\xi}(u,u)+2\RE\lgl y,\mathcal{D}_{\sharp\mu,\xi}^{\theta}\mathfrak{G}_{\mu}\mathrlap{\phantom{\mathfrak{F}}_{\mu}}\mathfrak{F}^{-1}(x)\rgl_{\omega}+\mathcal{I}^{f,\theta}\lc\mu,\mu,y^{\flat}\rc{}.
\end{align}

\noindent Note $\mathcal{D}_{\sharp\mu,\xi}^{\theta}\mathfrak{G}_{\mu}\mathrlap{\phantom{\mathfrak{F}}_{\mu}}\mathfrak{F}^{-1}(x)=\nabla \mathrlap{\phantom{\mathfrak{F}}_{\mu}}\mathfrak{F}^{-1}(x)$. Using $y\in\ker\nabla^{*}$, the latter implies

\begin{align}\label{EQ.PRP.RM_II_4}
\RE\lgl y,\mathcal{D}_{\sharp\mu,\xi}^{\theta}\mathfrak{G}_{\mu}\mathrlap{\phantom{\mathfrak{F}}_{\mu}}\mathfrak{F}^{-1}(x)\rgl_{\omega}=0.  
\end{align}

\noindent Equation \ref{EQ.PRP.RM_II_3} and Equation \ref{EQ.PRP.RM_II_4} show Equation \ref{EQ.PRP.RM_II_2}. Since $\sharp\mu>0$ in $A_{\xi}$, we further have $\mathcal{I}^{f,\theta}\lc\mu,\mu,y^{\flat}\rc{}=0$ if and only if $y=0$. This shows equivalence. Get $3)$.
\end{proof}

We know $T\vartheta(\xi)=\vartheta(\xi)\times I(\Delta_{\xi})^{\flat}$ by $2)$ in Proposition \ref{PRP.RM_Embedded_Submanifold_II}. Definition \ref{DFN.RM_III} gives smooth map $\Theta:T\vartheta(\xi)\longrightarrow B_{\xi}^{*}$. Proposition \ref{PRP.RM_III} shows evaluating the latter on square integrable absolutely continuous paths to $\vartheta(\xi)$ induces admissible paths. Their vector fields minimise energy along a given absolutely continuous path.

\begin{dfn}\label{DFN.RM_III}\hspace{1cm}
\begin{itemize}
\item[1)] We define $\Theta:T\vartheta(\xi)\longrightarrow B_{\xi}^{*}$ by setting

\begin{align}\label{EQ.DFN.RM_III_1}
\Theta(\mu,u):=\lc\mathfrak{G}_{\mu}\mathrlap{\phantom{\mathfrak{F}}_{\mu}}\mathfrak{F}^{-1}\lc\sharp u\rc\rc^{\flat}
\end{align}

\begin{reapply}
\end{reapply}

\noindent for all $\mu\in\vartheta(\xi)$ and $u\in I(\Delta_{\xi})^{\flat}$. 

\item[2)] For all absolutely continuous $\mu:[a,b]\longrightarrow\vartheta(\xi)$, set

\begin{align}\label{EQ.DFN.RM_III_2}
\Theta(\mu,\dot{\mu})(t):=\Theta\lc\mu(t),\dot{\mu}(t)\rc{}
\end{align}

\begin{reapply}
\end{reapply}

\noindent for a.e.~$t\in [0,1]$.
\end{itemize}
\end{dfn}

\begin{rem}\label{REM.RM}
Following $1)$ in Definition \ref{DFN.RM_III}, note Equation \ref{EQ.PRP.RM_II_1} yields

\begin{align}\label{EQ.REM.RM_1}
\mathcal{I}^{f,\theta}\lc\mu,\mu,\Theta(\mu,u)\rc{}=g_{\mu}^{\xi}(u,u)
\end{align}

\noindent for all $\mu\in\vartheta(\xi)$ and $u\in I(\Delta_{\xi})^{\flat}$. We use this throughout our discussion.
\end{rem}

\begin{prp}\label{PRP.RM_III}
We consider Riemannian manifold $(\vartheta(\xi),g^{\xi})$. Let $\mu:[a,b]\longrightarrow\vartheta(\xi)$ be absolutely continuous. If $\int_{a}^{b}\|\dot{\mu}(t)\|_{A^{*}}^{2}dt<\infty$, then

\begin{itemize}
\item[1)] $\lc\mu,\Theta(\mu,\dot{\mu})\rc\in\Admab\lc\mu\lc{}a\rc{},\mu\lc{}b\rc\rc$,

\item[2)] $E^{f,\theta}\lc\mu,\Theta(\mu,\dot{\mu})\rc{}=\int_{a}^{b}g_{\mu(t)}^{\xi}\lc\dot{\mu}(t),\dot{\mu}(t)\rc{}dt<\infty$,

\item[3)] $E^{f,\theta}\lc\mu,\Theta(\mu,\dot{\mu})\rc\leq E(\mu,w)$ for all $(\mu,w)\in\Admab\lc\mu\lc{}a\rc{},\mu\lc{}b\rc\rc$.
\end{itemize}

\noindent Furthermore, we have equality in $3)$ if and only if $w(t)=\Theta(\mu,\dot{\mu})(t)$ for a.e.~$t\in [a,b]$.
\end{prp}
\begin{proof}
Assume $\int_{a}^{b}\|\dot{\mu}(t)\|_{A^{*}}^{2}dt<\infty$. Note continuity by itself implies

\begin{align}\label{EQ.PRP.RM_III_1}
\sup_{t\in [0,1]}\hspace{0.025cm} \sigma\lc\mu(t)\rc^{-1}=\sup_{t\in [0,1]}\hspace{0.025cm} \dblv{}L_{\sharp\mu(t),\supp\xi}^{-1}\dblv{}<\infty.
\end{align}

\noindent All Banach space norms we consider here are equivalent by finite-dimensionality. Using $2)$ in Proposition \ref{PRP.RM_II}, Equation \ref{EQ.PRP.RM_III_1} yields $C>0$ s.t.~

\begin{align}\label{EQ.PRP.RM_III_2}
\mathcal{I}^{f,\theta}\lc\mu(t),\mu(t),\Theta(\mu,\dot{\mu})(t)\rc{}=g_{\mu(t)}^{\xi}\lc\dot{\mu}(t),\dot{\mu}(t)\rc\leq C\cdot \big\|\dot{\mu}(t)\big\|_{A^{*}}^{2}
\end{align}

\noindent for a.e.~$t\in [a,b]$. Using $5)$ in Theorem \ref{THM.QE_AF}, Equation \ref{EQ.PRP.RM_III_2} yields $C',C''>0$ s.t.~

\begin{align}\label{EQ.PRP.RM_III_3}
\int_{a}^{b}\dblv{}\Theta(\mu,\dot{\mu})(t)\dblv_{B^{*}}^{2}dt\leq C'\cdot \int_{a}^{b}g_{\mu(t)}^{\xi}\lc\dot{\mu}(t),\dot{\mu}(t)\rc{}dt\leq C''\cdot \int_{a}^{b}\big\|\dot{\mu}(t)\big\|_{A^{*}}^{2}dt<\infty.
\end{align}

\noindent Equation \ref{EQ.PRP.RM_III_3} shows $\Theta(\mu,\dot{\mu})$ is square integrable. We calculate

\begin{align}\label{EQ.PRP.RM_III_4}
\dot{\mu}(t)=\nabla^{*}\mathfrak{G}_{\mu}\mathrlap{\phantom{\mathfrak{F}}_{\mu}}\mathfrak{F}^{-1}\lc\dot{\mu}(t)\rc{}=\nabla^{*}\Theta(\mu,\dot{\mu})(t)    
\end{align}

\noindent for a.e.~$t\in [a,b]$. Equation \ref{EQ.PRP.RM_III_2} and Equation \ref{EQ.PRP.RM_III_4} show $1)$ and $2)$. We show $3)$. For all $(\mu,w)\in\Admab\lc\mu\lc{}a\rc{},\mu\lc{}b\rc\rc$, note $\sharp\dot{\mu}(t)=\nabla^{*}\sharp w(t)$ for a.e.~$t\in [a,b]$ by the continuity equation. Using $3)$ in Proposition \ref{PRP.RM_II}, the latter implies $3)$ at once.
\end{proof}

Theorem \ref{THM.RM} uses Lemma \ref{LEM.RM}. The latter shows minimising geodesics with marginals in $\vartheta(\xi)$ are suitably approximated by minimising geodesics in $\vartheta(\xi)$ without change of marginals. Corollary \ref{COR.RM} implies $\vartheta(\xi)\subset\mathcal{C}_{A}(\xi)$ is a geodesic subspace as per $2)$ in Definition \ref{DFN.Metric_Functional_EVI_CNV}. The statement of Lemma \ref{LEM.RM} is more general. We show $1)$ in the lemma by extending convolution with Dirac sequences \cite{BK.Eva.2010.Partial_Differential_Equations} to the AF-$C^{*}$-setting. We show $2)$ in the lemma by adapting the proof of Lemma 3.30 in \cite{ART.Maa.2011.Discrete_OT_Markov}.\par
Lemma \ref{LEM.RM} uses the convolution of bounded Bochner measurable maps to $A_{\xi}^{*}$ with smooth maps on $\mathbb{R}$ having integrable first derivative. Definition \ref{DFN.Bochner_Convolution} gives such Bochner convolutions. Note Remark \ref{REM.Bochner_Convolution_I} and Remark \ref{REM.Bochner_Convolution_II}.

\begin{dfn}\label{DFN.Bochner_Convolution}\hspace{1cm}
\begin{itemize}
\item[1)] Set $C^{\infty,1}(\mathbb{R}):=\lset\varphi\in C^{\infty}(\mathbb{R})\ \vset\ \forall k\in\mathbb{N}:\ \frac{d^{k}}{dt^{k}}\varphi\in L^{1}(\mathbb{R})\rset$. For all closed intervals $I\subset\mathbb{R}$, we\linebreak say that a Bochner measurable map $\eta:I\longrightarrow A^{*}$ \cite{BK.Ion_Ion.1969.Vector_Valued_Lp_Duals} is bounded measurable if $\|\eta\|_{\infty}:=\esssup_{t\in I}\|\eta(t)\|_{A^{*}}<\infty$.

\item[2)] Let $\eta:\mathbb{R}\longrightarrow A_{\xi}^{*}$ be bounded measurable. For all $\varphi\in C^{\infty,1}(\mathbb{R})$, we define the Bochner convolution map $\eta\ast\varphi:\mathbb{R}\longrightarrow A_{\xi}^{*}$ by setting

\vspace{-0.01375cm}
\begin{align}\label{EQ.DFN.Bochner_Convolution_1}
\lc\eta\ast\varphi\rc{}(t):=\int_{-\infty}^{\infty}\eta(s)\varphi\lc{}t-s\rc{}ds
\end{align}

\vspace{-0.0125cm}
\noindent for all $t\in\mathbb{R}$.
\end{itemize}
\end{dfn} 

\begin{rem}\label{REM.Bochner_Convolution_I}
In the finite-dimensional setting, Bochner integration specialises to one-dimensional analogues in components. Let $\eta:\mathbb{R}\longrightarrow A^{*}$ be bounded measurable. For all $\varphi\in C^{\infty,1}(\mathbb{R})$, the map $s\mapsto\eta(s)\varphi\lc{}t-s\rc$ is indeed integrable for all $t\in\mathbb{R}$.
\end{rem}

Let $\eta:\mathbb{R}\longrightarrow A_{\xi}^{*}$ be bounded measurable and $\varphi\in C^{\infty,1}(\mathbb{R})$. For all $x\in A$, we consider the map $s\mapsto\eta_{x}(s):=\eta(s)(x)$ and have

\begin{align}\label{EQ.SSEC.QOT_AC_RM_4}
\lc\eta\ast\varphi\rc{}(t)^{\flat}(x)=\int_{-\infty}^{\infty}\eta(s)^{\flat}(x)\varphi\lc{}t-s\rc{}ds=\lc\eta_{x}\ast\varphi\rc{}(t)
\end{align}

\noindent for all $t\in\mathbb{R}$. Equation \ref{EQ.SSEC.QOT_AC_RM_4} shows standard results for convolutions apply \cite{BK.Eva.2010.Partial_Differential_Equations}. We have $\|\eta\ast\varphi\|_{\infty}\leq \|\eta\|_{\infty}\|\varphi\|_{1}$ by H\"older. For all $k\in\mathbb{N}$, we moreover have

\begin{align}\label{EQ.SSEC.QOT_AC_RM_5}
\frac{d^{k}}{dt^{k}}\lc\eta\ast\varphi\rc{}(t)=\bigg(\eta\ast\frac{d^{k}}{dt^{k}}\varphi\bigg)(t)
\end{align}

\noindent for all $t\in\mathbb{R}$. If $\eta$ is $t$-a.e.~differentiable and $\dot{\eta}$ bounded measurable, then

\begin{align}\label{EQ.SSEC.QOT_AC_RM_6}
\frac{d}{dt}\lc\eta\ast\varphi\rc{}(t)=\lc\dot{\eta}\ast\varphi\rc{}(t)    
\end{align}

\noindent for a.e.~$t\in\mathbb{R}$.

\begin{rem}\label{REM.Bochner_Convolution_II}
For all bounded measurable $\eta:\mathbb{R}\longrightarrow A_{\xi}^{*}$, we have bounded measurable $h^{\perp}(\eta):\mathbb{R}\longrightarrow A_{\xi}^{*}$ by setting $h^{\perp}(\eta)(t):=h^{\perp}\lc\eta(t)\rc$ for all $t\in\mathbb{R}$.
\end{rem}


\pagebreak


Let $\mu:[0,1]\longrightarrow\vartheta(\xi)$ be absolutely continuous. We extend to bounded measurable $\mu:\mathbb{R}\longrightarrow\vartheta(\xi)$ by setting $\mu(t):=\xi$ if $t\notin [0,1]$. Thus $h^{\perp}(\mu)(t)=0$ if $t\notin [0,1]$, hence $h^{\perp}(\mu)$ is bounded measurable with compact support in $[0,1]$. Assume $\|\dot{\mu}\|_{\infty}<\infty$ and $\varphi\in C^{\infty,1}(\mathbb{R})$ s.t.~$\varphi\geq 0$ and $\|\varphi\|_{1}=1$. For all $\eta\in\vartheta(\xi)$, get $\sharp\eta>0$ in $A_{\xi}$ by $1)$ in Proposition \ref{PRP.RM_Embedded_Submanifold_II}. Since further $\xi\in\vartheta(\xi)$ by $2)$ in Theorem \ref{THM.Wstar_Derivation_QG_HSG_Regularity}, continuity implies

\begin{align}\label{EQ.SSEC.QOT_AC_RM_7}
\inf_{t\in\mathbb{R}}\hspace{0.025cm} \sigma\lc\mu(t)\rc{}>0.
\end{align}

\noindent Using $\varphi\geq 0$ and $\|\varphi\|_{1}=1$, Equation \ref{EQ.SSEC.QOT_AC_RM_4} and Equation \ref{EQ.SSEC.QOT_AC_RM_7} show

\begin{align}\label{EQ.SSEC.QOT_AC_RM_8}
\sharp\lc\mu\ast\varphi\rc{}(t)\geq\inf_{t\in\mathbb{R}}\hspace{0.025cm} \sigma\lc\mu(t)\rc\cdot \supp\xi>0
\end{align}

\noindent in $A_{\xi}$ for all $t\in\mathbb{R}$. Equation \ref{EQ.SSEC.QOT_AC_RM_8} shows $\lc\mu\ast\varphi\rc{}(t)\in\vartheta(\xi)$ for all $t\in\mathbb{R}$. Taken together with Equation \ref{EQ.SSEC.QOT_AC_RM_5}, we have smooth $\mu\ast\varphi:\mathbb{R}\longrightarrow\vartheta(\xi)$. Equation \ref{EQ.SSEC.QOT_AC_RM_6} shows

\begin{align}\label{EQ.SSEC.QOT_AC_RM_9}
\frac{d}{dt}\lc\mu\ast\varphi\rc{}(t)=\frac{d}{dt}\lc{}h^{\perp}(\mu)\ast\varphi\rc{}(t)=\lc\dot{\mu}\ast\varphi\rc{}(t)\in I(\Delta_{\xi})^{\flat}
\end{align}

\noindent for a.e.~$t\in\mathbb{R}$.

\begin{rem}\label{REM.RM_Heat_Kernel}
For all $n\in\mathbb{N}$, we consider normal distribution for $\sigma^{2}=n^{-1}$ given by

\begin{align}\label{EQ.REM.RM_Heat_Kernel_1}
\varphi_{n}(t):=\sqrt{\frac{n}{2\pi}}\hspace{0.025cm}\exp\bigg(-\frac{t^{2}n}{2}\bigg)
\end{align}

\noindent for all $t\in\mathbb{R}$ \cite{BK.Pap.2002.Measures}. We have $\varphi_{n}\in C^{\infty,1}(\mathbb{R})$, $\varphi_{n}\geq 0$ and $\dblv{}\varphi_{n}\dblv_{1}=1$ in each case. We use such Dirac sequence $\{\varphi_{n}\}_{n\in\mathbb{N}}\subset C^{\infty,1}(\mathbb{R})$ \cite{BK.Eva.2010.Partial_Differential_Equations} for Bochner convolutions in Lemma \ref{LEM.RM}.
\end{rem}

\begin{lem}\label{LEM.RM}
Let $\mu^{0},\mu^{1}\in\vartheta(\xi)$ and $(\mu,w)\in\Admnullone\lc\mu^{0},\mu^{1}\rc$ s.t.~$E^{f,\theta}(\mu,w)<\infty$.

\begin{itemize}
\item[1)] If $\mu:[0,1]\longrightarrow\vartheta(\xi)$ and $\|\dot{\mu}\|_{\infty}<\infty$, then there exists family $\lset\mu^{n}:[0,1] \longrightarrow\vartheta(\xi)\rset_{n\in\mathbb{N}}$ of smooth paths s.t.~

\begin{itemize}
\item[1.1)] $\lc\mu^{n},\Theta\lc\mu^{n},\dot{\mu}^{n}\rc\rc\in\Admnullone$ for all $n\in\mathbb{N}$, \phantom{\big)}

\item[1.2)] $\lim_{n\in\mathbb{N}}\hspace{0.025cm} \lc\mu^{n},\Theta\lc\mu^{n},\dot{\mu}^{n}\rc\rc{}=\lc\mu,\Theta(\mu,\dot{\mu})\rc$ in $\Admnullone$, \phantom{\big)}

\item[1.3)] $\lim_{n\in\mathbb{N}}E^{f,\theta}\lc\mu^{n},\Theta\lc\mu,\dot{\mu}^{n}\rc\rc{}=E^{f,\theta}\lc\mu,\Theta(\mu,\dot{\mu})\rc\leq E^{f,\theta}(\mu,w)$. \phantom{\big)}
\end{itemize}

\begin{reapply}
\end{reapply}

\item[2)] If $\|\dot{\mu}\|_{\infty}<\infty$, then there exists $(\mu^{n},w^{n})_{n\in\mathbb{N}}\subset\Adm\lc\mu^{0},\mu^{1}\rc$ and $C>0$ s.t.~

\begin{itemize}
\item[2.1)] $\mu^{n}:[0,1]\longrightarrow\vartheta(\xi)$ and $\|\dot{\mu}^{n}\|_{\infty}\leq C\|\dot{\mu}\|_{\infty}$ for all $n\in\mathbb{N}$, \phantom{\big)}

\item[2.2)] $\liminf_{n\in\mathbb{N}}E^{f,\theta}(\mu^{n},w^{n})\leq E^{f,\theta}(\mu,w)$. \phantom{\big)}
\end{itemize}

\begin{reapply}
\end{reapply}

\item[3)] If $(\mu,w)\in\Geo\lc\mu^{0},\mu^{1}\rc$, then $\|\dot{\mu}\|_{\infty}<\infty$.
\end{itemize}
\end{lem}


\pagebreak


\begin{proof}
We show $1)$. Assume its setting. In particular, we have $\lc\mu,\Theta(\mu,\dot{\mu})\rc\in\Admnullone$ and $E^{f,\theta}\lc\mu,\Theta(\mu,\dot{\mu})\rc\leq E^{f,\theta}(\mu,w)$ by $1)$, resp.~$3)$ in Proposition \ref{PRP.RM_III}. Continuity implies $\|\mu\|_{\infty}=\esssup_{t\in [0,1]}\|\mu\|_{A^{*}}<\infty$.\par
For all $n\in\mathbb{N}$, let $\varphi_{n}\in C^{1,\infty}(\mathbb{R})$ be the normal distribution as per Equation \ref{EQ.REM.RM_Heat_Kernel_1} and set $\mu^{n}:=\mu\ast\varphi_{n}$. Then $\|\mu\|_{\infty},\|\dot{\mu}\|_{\infty}<\infty$ implies $\mu^{n}:\mathbb{R}\longrightarrow\vartheta(\xi)$ is smooth s.t.~$\dot{\mu}^{n}=\dot{\mu}\ast\varphi_{n}$ in each case. For all $n\in\mathbb{N}$, we directly verify

\begin{align}\label{EQ.LEM.RM_1}
\|\mu^{n}\|_{\infty}\leq \|\mu\|_{\infty}<\infty,\ \|\dot{\mu}^{n}\|_{\infty}\leq \|\dot{\mu}\|_{\infty}<\infty
\end{align}

\noindent using $\|\varphi\|_{1}=1$. We show $\lset\mu^{n}\vert_{[0,1]}:[0,1] \longrightarrow\vartheta(\xi)\rset_{n\in\mathbb{N}}$ is a sequence as claimed. Get $1.1)$ by $1)$ in Proposition \ref{PRP.RM_III}. Testing on $A$, standard properties of Dirac sequences imply $\mu(t)=w^{*}$-$\lim_{n\in\mathbb{N}}\mu^{n}(t)$ for all $t\in [0,1]$ and $\dot{\mu}(t)=w^{*}$-$\lim_{n\in\mathbb{N}}\dot{\mu}^{n}(t)$ for a.e.~$t\in [0,1]$ \cite{BK.Eva.2010.Partial_Differential_Equations}. All norms and operator topologies here are equivalent by finite-dimensionality. Thus $1.2)$ and $1.3)$ follow by dominated convergence if there exists $C>0$ s.t.~

\begin{align}\label{EQ.LEM.RM_2}
\dblv{}\sharp\Theta\lc\mu^{n}(t),\dot{\mu}^{n}(t)\rc\dblv_{\omega},g_{\mu^{n}(t)}^{\xi}\lc\dot{\mu}(t),\dot{\mu}(t)\rc\leq C
\end{align}

\noindent for a.e.~$t\in [0,1]$ and all $n\in\mathbb{N}$.\par
We show there exists $C>0$ as for Equation \ref{EQ.LEM.RM_2}. Using Equation \ref{EQ.LEM.RM_1}, applying $\|\dot{\mu}^{n}\|_{\infty}\leq \|\dot{\mu}\|_{\infty}$ in each case lets us estimate

\begin{align}\label{EQ.LEM.RM_3}
\dblv{}\sharp\Theta\lc\mu^{n}(t),\dot{\mu}^{n}(t)\rc\dblv_{\omega}\leq \dblv{}\mathfrak{G}_{\mu^{n}(t)}\dblv{}\cdot \dblv{}\mathfrak{F}_{\mu^{n}(t)}^{-1}\dblv{}\cdot \|\dot{\mu}\|_{\infty} 
\end{align}

\noindent and

\begin{align}\label{EQ.LEM.RM_4}
g_{\mu^{n}(t)}^{\xi}\lc\dot{\mu}^{n}(t),\dot{\mu}^{n}(t)\rc\leq \dblv{}\mathfrak{F}_{\mu^{n}(t)}^{-\frac{1}{2}}\dblv{}\cdot \|\dot{\mu}\|_{\infty} 
\end{align}

\noindent for a.e.~$t\in [0,1]$ and all $n\in\mathbb{N}$. Since moreover $\|\mu^{n}\|_{\infty}\leq \|\mu\|_{\infty}$ in each case, get uniform bound for $\lset\dblv{}\mathfrak{G}_{\mu^{n}(t)}\dblv{}\rset_{t\in\mathbb{R},n\in\mathbb{N}}$ by continuity. Uniform bound for $\lset\dblv{}\mathfrak{F}_{\mu^{n}(t)}^{-1}\dblv{}\rset_{t\in\mathbb{R},n\in\mathbb{N}}$ follows by $1)$ in Proposition \ref{PRP.RM_I} if

\begin{align}\label{EQ.LEM.RM_5}
\inf_{t\in\mathbb{R}}\hspace{0.025cm} \sigma\lc\mu^{n}(t)\rc\geq\inf_{t\in\mathbb{R}}\hspace{0.025cm} \sigma\lc\mu(t)\rc{}>0
\end{align}

\noindent for all $n\in\mathbb{N}$. Using Lemma \ref{LEM.Spectral_Gap_Maximal}, Equation \ref{EQ.SSEC.QOT_AC_RM_8} shows Equation \ref{EQ.LEM.RM_5} by maximality of spectral gaps. Applying uniform bounds to Equation \ref{EQ.LEM.RM_3} and Equation \ref{EQ.LEM.RM_4} yields $C>0$ as required. Get $1.2)$ and $1.3)$ by dominated convergence. Altogether, get $1)$.\par


\pagebreak


We show $2)$. Assume $\|\dot{\mu}\|_{\infty}<\infty$. We adapt the proof of Lemma 3.30 in \cite{ART.Maa.2011.Discrete_OT_Markov}. We construct two types of perturbed paths and concatenate them. For all $\varepsilon\in (0,1)$, set

\begin{align}\label{EQ.LEM.RM_6}
\mu^{\varepsilon}(t):=(1-\varepsilon)\mu(t)+\varepsilon\xi,\ v^{\varepsilon}(t):=(1-\varepsilon)w(t)
\end{align}

\noindent for all $t\in [0,1]$. Since $\sharp\mu^{0},\sharp\mu^{1},\sharp\xi>0$ in $A_{\xi}$, we see $\mu^{\varepsilon}(t)\in\vartheta(\xi)$ in each case. Moreover, we directly verify $\lc\mu^{\varepsilon},v^{\varepsilon}\rc\in\Admnullone$. This is the first type of perturbed path.\par
For all $\varepsilon\in (0,1)$ and $k\in\lset{}0,1\rset$, set

\begin{align}\label{EQ.LEM.RM_7}
\mu^{k,\varepsilon}(t):=\lc{}1-t\rc\mu(k)+t\mu^{\varepsilon}(k)
\end{align}

\noindent for all $t\in [0,1]$. Since $\sharp\mu^{0},\sharp\mu^{1},\sharp\xi>0$ in $A_{\xi}$, we see $\mu^{k,\varepsilon}(t)\in\vartheta(\xi)$ in each case. There further exists $C>0$ s.t.~

\begin{align}\label{EQ.LEM.RM_8}
\sharp\mu^{k,\varepsilon}(t)\geq C\cdot \supp\xi
\end{align}

\noindent for all $\varepsilon\in (0,1)$, $k\in\lset{}0,1\rset$ and $t\in [0,1]$. For all $\varepsilon\in (0,1)$ and $k\in\lset{}0,1\rset$, set

\begin{align}\label{EQ.LEM.RM_9}
v^{k,\varepsilon}(t):=\varepsilon\cdot \Theta\big(\mu^{k,\varepsilon}(t),\xi-\mu(k)\big)
\end{align}

\noindent for all $t\in [0,1]$. Since $\frac{d}{dt}\mu^{k,\varepsilon}(t)=\varepsilon\lc\xi-\mu(k)\rc$ in each case, get $\lc\mu^{k,\varepsilon},v^{k,\varepsilon}\rc\in\Adm\lc\mu^{0},\mu^{1}\rc$ at once by $1)$ in Proposition \ref{PRP.RM_III}. This is the second type of perturbed path.\par
We concatenate these two types of paths. For all $\varepsilon\in (0,1)$, we define concatenated path on $[0,1]$ by setting

\begin{align*}
\lc\mu^{\varepsilon},w^{\varepsilon}\rc{}(t):=
\begin{cases} 
\lc\mu^{0,\varepsilon},\varepsilon^{-1}v^{0,\varepsilon}\rc\lc\varepsilon^{-1}t\rc & \If\ t\leq\varepsilon, \phantom{\bigg)} \\
\lc\mu^{\varepsilon},\lc{}1-2\varepsilon\rc^{-1}v^{\varepsilon}\rc\lc (1-2\varepsilon)^{-1}(t-\varepsilon)\rc & \If\ \varepsilon<t<1-\varepsilon, \phantom{\big)} \\
\lc\mu^{1,\varepsilon},\varepsilon^{-1}v^{1,\varepsilon}\rc\lc\varepsilon^{-1}(1-t)\rc & \If\ t\geq 1-\varepsilon. \phantom{\bigg)}
\end{cases}
\end{align*}

\noindent We have $\mu^{\varepsilon}:[0,1]\longrightarrow\vartheta(\xi)$ and $\lc\mu^{\varepsilon},w^{\varepsilon}\rc\in\Adm\lc\mu^{0},\mu^{1}\rc$ in each case. Moreover, we directly verify there exists $C>0$ s.t.~$\sup_{\varepsilon\in (0,1)}\|\dot{\mu}^{\varepsilon}\|_{\infty}\leq C\|\dot{\mu}\|_{\infty}$. We readily see $2.1)$ is satisfied for all countable subsequences of $\lc\mu^{\varepsilon},w^{\varepsilon}\rc_{\varepsilon>0}$. We claim $2.2)$ is likewise satisfied.\par
We show $2.2)$. For all $\varepsilon\in (0,1)$ and $k\in\lset{}0,1\rset$, joint convexity of quasi-entropies as per $1)$ in Theorem \ref{THM.QE_AF} shows

\begin{align}\label{EQ.LEM.RM_10} 
E^{f,\theta}\lc\mu^{\varepsilon},v^{\varepsilon}\rc\leq (1-\varepsilon)\cdot E^{f,\theta}(\mu,w). 
\end{align}

\noindent Using Lemma \ref{LEM.Spectral_Gap_Maximal}, note Equation \ref{EQ.LEM.RM_8} shows $\inf_{t\in [0,1]}\sigma\lc\mu^{k,\varepsilon}(t)\rc\geq C$ in each case by maximality of spectral gaps. We invert the latter to get $\sup_{t\in [0,1]}\sigma\lc\mu^{k,\varepsilon}(t)\rc^{-1}\leq C^{-1}$.\par
Using $2)$ in Proposition \ref{PRP.RM_II}, we therefore have

\begin{align}\label{EQ.LEM.RM_11} 
E^{f,\theta}\lc\mu^{k,\varepsilon},v^{k,\varepsilon}\rc\leq\varepsilon^{2}\sigma(\Delta)^{-1}C^{-\theta}\cdot \big\|\sharp\xi-\sharp\mu(k)\big\|_{\tau}^{2}
\end{align}

\noindent for all $\varepsilon\in (0,1)$ and $k\in\lset{}0,1\rset$. Rescaling as per Remark \ref{REM.Energy_Functional_Reparametrisation} shows

\begin{align}\label{EQ.LEM.RM_12}
E^{f,\theta}\lc\mu^{\varepsilon},w^{\varepsilon}\rc{}=\varepsilon^{-1}E^{f,\theta}\lc\mu^{0,\varepsilon},v^{0,\varepsilon}\rc{}+\lc{}1-2\varepsilon\rc^{-1}E^{f,\theta}\lc\mu^{\varepsilon},v^{\varepsilon}\rc{}+\varepsilon^{-1}E^{f,\theta}\lc\mu^{1,\varepsilon},v^{1,\varepsilon}\rc{}
\end{align}

\noindent for all $\varepsilon\in (0,1)$. Applying Equation \ref{EQ.LEM.RM_10} and Equation \ref{EQ.LEM.RM_11} to Equation \ref{EQ.LEM.RM_12} yields

\begin{align}\label{EQ.LEM.RM_13}
E^{f,\theta}\lc\mu^{\varepsilon},w^{\varepsilon}\rc\leq 2\varepsilon\cdot \sigma(\Delta)^{-1}C^{-\theta}\cdot \big\|\sharp\xi-\sharp\mu(k)\big\|_{\tau}^{2}+\frac{1-\varepsilon}{1-2\varepsilon}E^{f,\theta}(\mu,w)
\end{align}

\noindent for all $\varepsilon\in (0,1)$. Letting $\varepsilon\downarrow 0$ in Equation \ref{EQ.LEM.RM_13} yields $E^{f,\theta}(\mu,w)$ on its right-hand side. We therefore have $2.2)$ as claimed. Altogether, get $2)$.\par
We show $3)$. Assume $\mu\in\Geo\lc\mu^{0},\mu^{1}\rc$. Minimising geodesics have $t$-a.e.~constant speed by $1)$ in Proposition \ref{PRP.QOT_Distance_Geodesics}. By definition, the quasi-entropy evaluated on $(\mu,w)$ is thus $t$-a.e.~constant. There exists $C>0$ s.t.~$\sup_{t\in [0,1]}\|\sharp w(t)\|_{\omega}\leq C$ by $4)$ in Theorem \ref{THM.QE_AF}. We have $\|\nabla^{*}\|<\infty$ by finite-dimensionality. The continuity equation lets us calculate

\begin{align}\label{EQ.LEM.RM_14}
\big\|\dot{\mu}(t)\big\|_{\tau}=\dblv{}\nabla^{*}\sharp w(t)\dblv_{\tau}\leq\big\|\nabla^{*}\big\|\cdot \dblv{}\sharp w(t)\dblv_{\omega}\leq\big\|\nabla^{*}\big\|\cdot C<\infty
\end{align}

\noindent for a.e.~$t\in [0,1]$. Equation \ref{EQ.LEM.RM_14} implies $3)$ at once.
\end{proof}

\begin{thm}\label{THM.RM}
Let $(\phi,\bpsi,\gamma,\nabla)$ be noncommutative differential structure for tracial AF-$C^{*}$-algebras $(A,\tau)$ and $(B,\omega)$ in $\lc{}f,\theta\rc$-setting. Assume $A$ and $B$ are finite-dimensional. If $\xi\in\SII(A)$ is a fixed state, then $\mathcal{W}_{\nabla\vert\vartheta(\xi)\times\vartheta(\xi)}^{f,\theta}$ is the distance induced by $g^{\xi}$.
\end{thm}
\begin{proof}
Let $\xi\in\SII(A)$ be a fixed state. Proposition \ref{PRP.RM_III} shows the induced distance $d^{\xi}$ of $g^{\xi}$ is given by minimising

\begin{align}\label{EQ.THM.RM_1}
\sqrt{E^{f,\theta}\lc\mu,\Theta(\mu,\dot{\mu})\rc{}}=\sqrt{\int_{0}^{1}g_{\mu(t)}^{\xi}\lc\dot{\mu}(t),\dot{\mu}(t)\rc{}dt}
\end{align}

\noindent over smooth paths $\mu:[0,1]\longrightarrow\vartheta(\xi)$. Thus $1)$ and $2)$ in Lemma \ref{LEM.RM} show $d^{\xi}$ is given by minimising over absolutely continuous path with marginals in $\vartheta(\xi)$ and bounded measurable derivative, hence we conclude by $3)$ in Lemma \ref{LEM.RM}.
\end{proof}

\begin{cor}\label{COR.RM}
For all $\mu^{0},\mu^{1}\in\vartheta(\xi)$, there exists $(\mu,w)\in\Geo\lc\mu^{0},\mu^{1}\rc$ s.t.~$\mu(t)\in\vartheta(\xi)$ for all $t\in [0,1]$ and $\mu:[0,1]\longrightarrow\vartheta(\xi)$ is a minimising geodesic in distance induced by $g^{\xi}$.
\end{cor}
\begin{proof}
Let $\mu^{0},\mu^{1}\in\vartheta(\xi)$. Get $(\mu,w)\in\Geo\lc\mu^{0},\mu^{1}\rc$ by $3)$ in Corollary \ref{COR.QOT_Distance_AC_II}. Lemma \ref{LEM.RM} implies $\mu(t)\in\vartheta(\xi)$ for all $t\in [0,1]$ by minimality. We conclude by Theorem \ref{THM.RM}.
\end{proof}


\subsubsection*{Accessibility components of square integrable normal states}

Assuming spectral gaps of quantum Laplacians and fixed parts, Theorem \ref{THM.QOT_Distance_AC_L2} classifies accessibility components of square integrable normal states by showing each one is a norm closed convex subsets of all such states with identical fixed part. Theorem \ref{THM.QOT_Distance_AC_L2} uses Lemma \ref{LEM.QOT_Distance_AC_L2}. We show the lemma by twice reduction. This lets us adapt the proof of Proposition 9.2 in \cite{ART.Car_Maa.2020.Quantum_OT_III}. In the finite-dimensional setting, assumptions as above are satisfied and Corollary \ref{COR.QOT_Distance_AC_L2} classifies all accessibility components using fixed parts.\par
Moreover, the coarse graining process reveals more general classification schemes by intersecting with convex subsets of states other than square integrable normal ones. In the logarithmic mean setting and assuming strictly positive lower Ricci bounds, as well as finitely supported fixed part but not spectral gaps, Theorem \ref{THM.QOT_Distance_AC_Rel_Ent} classifies accessibility components of normal states with finite quantum relative entropy using fixed parts. Here, Example \ref{BSP.QOT_Distance_AC_L2_2} constructs quantum Laplacians having spectral gaps for the unique hyperfinite type II$_{1}$-factor.\par
Let $(\phi,\bpsi,\gamma,\nabla)$ be noncommutative differential structure for tracial AF-$C^{*}$-algebras $(A,\tau)$ and $(B,\omega)$ in $\lc{}f,\theta\rc$-setting. Assume $\sigma(\Delta)>0$.

\begin{lem}\label{LEM.QOT_Distance_AC_L2}
Let $\xi\in\mathcal{S}_{-1}^{\NI,2}(A_{\xi})$ be a fixed state. For all $\mu,\eta\in\Fix_{A}(\xi)\cap\mathcal{S}^{\NI,2}(A)$ and $\varepsilon\in (0,1]$, we have

\begin{align}\label{EQ.LEM.QOT_Distance_AC_L2_1}
\mathcal{W}_{\nabla}^{f,\theta}(\mu,\eta)\leq 2\sigma(\Delta)^{-\frac{1}{2}}\sigma(\xi)^{-\frac{\theta}{2}}\varepsilon^{-\frac{\theta}{2}}\cdot \bigg(\dblv{}\varepsilon h^{\perp}\lc\sharp\mu\rc\dblv_{\tau}+\dblv{}(1-\varepsilon)h^{\perp}\lc\sharp\mu\rc{}-h^{\perp}\lc\sharp\eta\rc\dblv_{\tau}\bigg)<\infty.
\end{align}
\end{lem}
\begin{proof}
We reduce twice in order to adapt the proof of Proposition 9.2 in \cite{ART.Car_Maa.2020.Quantum_OT_III}. First, we reduce to $\mu,\eta\in\Fix_{A}(\xi)\cap\mathcal{S}^{\NI,2}(A)$ s.t.~$\bar{\mu}_{j},\bar{\eta}_{j}\in\vartheta\lc\bar{\xi}_{j}\rc$ for a.e.~$j\in\mathbb{N}$. Secondly, we reduce to the finite-dimensional setting. Let $\varepsilon\in (0,1]$. Set 

\begin{align}\label{EQ.LEM.QOT_Distance_AC_L2_2}
C_{\varepsilon}:=2\sigma(\Delta)^{-\frac{1}{2}}\sigma(\xi)^{-\frac{\theta}{2}}\varepsilon^{-\frac{\theta}{2}}.    
\end{align}

We engage in the first reduction. Let $\mu,\eta\in\Fix_{A}(\xi)\cap\mathcal{S}^{\NI,2}(A)$. Note $\mathcal{W}_{\nabla}^{f,\theta}$ is l.s.c.~in $w^{*}$-topology by $3)$ in Theorem \ref{THM.QOT_Distance}. In addition, we know $2.1)$ in Proposition \ref{PRP.Wstar_Derivation_QG_HSG_II} ensures $h:[0,\infty]\longrightarrow\BII(A^{*})$ is $w^{*}$-continuous on $\SII(A)$. We obtain

\begin{align}\label{EQ.LEM.QOT_Distance_AC_L2_3}
\mathcal{W}_{\nabla}^{f,\theta}(\mu,\eta)\leq\liminf_{t\downarrow 0}\hspace{0.0675cm} \mathcal{W}_{\nabla}^{f,\theta}\lc{}h_{t}(\mu),h_{t}(\eta)\rc{}.
\end{align}

\noindent Strong continuity of $h:[0,\infty)\longrightarrow\BII\lc{}L^{2}(A,\tau)\rc$ as per $1)$ in Proposition \ref{PRP.Wstar_Derivation_QG_HSG_II} together with $\lb{}h_{t},h^{\perp}\rb{}=0$ for all $t\geq 0$ further yields

\begin{align}\label{EQ.LEM.QOT_Distance_AC_L2_4}
\dblv{}\varepsilon h^{\perp}\lc\sharp\mu\rc\dblv_{\tau}=\lim_{t\downarrow 0}\hspace{0.025cm} \dblv{}\varepsilon h^{\perp}\lc\sharp h_{t}(\mu)\rc\dblv_{\tau}  
\end{align}

\noindent and

\begin{align}\label{EQ.LEM.QOT_Distance_AC_L2_5}
\dblv{}(1-\varepsilon)h^{\perp}\lc\sharp\mu\rc{}-h^{\perp}\lc\sharp\eta\rc\dblv_{\tau}=\lim_{t\downarrow 0}\hspace{0.025cm} \dblv{}(1-\varepsilon)h^{\perp}\lc\sharp h_{t}(\mu)\rc{}-h^{\perp}\lc\sharp h_{t}(\eta)\rc\dblv_{\tau}.  
\end{align}


\pagebreak


Equation \ref{EQ.LEM.QOT_Distance_AC_L2_3}, Equation \ref{EQ.LEM.QOT_Distance_AC_L2_4} and Equation \ref{EQ.LEM.QOT_Distance_AC_L2_5} imply Equation \ref{EQ.LEM.QOT_Distance_AC_L2_1} if

\begin{align}\label{EQ.LEM.QOT_Distance_AC_L2_6}
\mathcal{W}_{\nabla}^{f,\theta}\lc{}h_{t}(\mu),h_{t}(\eta)\rc\leq C_{\varepsilon}\cdot \bigg(\dblv{}\varepsilon h^{\perp}\lc\sharp h_{t}(\mu)\rc\dblv_{\tau}+\dblv{}(1-\varepsilon)h^{\perp}\lc\sharp h_{t}(\mu)\rc{}-h^{\perp}\lc\sharp h_{t}(\eta)\rc\dblv_{\tau}\bigg)
\end{align}

\noindent for all $t>0$. If $\xi_{j}\neq 0$ for $j\in\mathbb{N}$, then $\overline{h_{t}(\mu)}_{j}=h_{t}\lc\bar{\mu}_{j}\rc{},\overline{h_{t}(\eta)}_{j}=h_{t}\lc\bar{\eta}_{j}\rc\in\vartheta\lc\bar{\xi}_{j}\rc$ for all $t>0$ by $1.3)$ in Proposition \ref{PRP.Wstar_Derivation_QG_HSG_Fixed_Part_I} and $1.3)$ in Theorem \ref{THM.Wstar_Derivation_QG_HSG_Regularity}. Since $\xi_{j}\neq 0$ for a.e.~$j\in\mathbb{N}$, we see Equation \ref{EQ.LEM.QOT_Distance_AC_L2_6} lets us apply the first reduction by $3)$ in Theorem \ref{THM.QOT_Distance}.\par
We engage in the second reduction. Let $\mu,\eta\in\Fix_{A}(\xi)\cap\mathcal{S}^{\NI,2}(A)$. Assume there exists $k\in\mathbb{N}$ s.t.~$\bar{\mu}_{j},\bar{\eta}_{j}\in\vartheta\lc\bar{\xi}_{j}\rc$ for all $j\geq k$ in $\mathbb{N}$. Let $0<\delta<\sigma(\xi)$. Set

\begin{align}\label{EQ.LEM.QOT_Distance_AC_L2_7}
C_{\delta}:=\sigma(\xi)-\delta.    
\end{align}

\noindent Following Remark \ref{REM.AF_Support_Projection}, $1)$ in Lemma \ref{LEM.AF_Support_Projection_II} implies there exists $l\in\mathbb{N}$ s.t.~

\begin{align}\label{EQ.LEM.QOT_Distance_AC_L2_8}
0<C_{\delta}\leq\sigma\lc\bar{\xi}_{j}\rc{}
\end{align}

\noindent for all $j\geq l$ in $\mathbb{N}$. Set $m:=\max\hspace{-0.033875cm} \lset{}k,l\rset$. Further set

\begin{align}\label{EQ.LEM.QOT_Distance_AC_L2_9}
\mu^{\varepsilon}:=(1-\varepsilon)\mu+\varepsilon\xi
\end{align}

\noindent for all $t\in [0,1]$. For all $j\in\mathbb{N}$, set $\bar{\mu}_{j}^{\varepsilon}:=\mu^{\varepsilon}(1_{A_{j}})^{-1}\mu_{j}^{\varepsilon}$ as per $1)$ in Definition \ref{DFN.AF_Cstar_Trace_Dualisation_Paths}.\par
Assume

\begin{align}\label{EQ.LEM.QOT_Distance_AC_L2_10}
\mathcal{W}_{\nabla_{\hspace{-0.055cm} j}}^{f,\theta}\lc\bar{\mu}_{j},\bar{\mu}_{j}^{\varepsilon}\rc\leq 2\sigma\lc\Delta_{j}\rc^{-\frac{1}{2}}C_{\delta}^{-\frac{\theta}{2}}\varepsilon^{-\frac{\theta}{2}}\dblv{}\varepsilon h^{\perp}\lc\sharp\bar{\mu}_{j}\rc\dblv_{\tau}    
\end{align}

\noindent and

\begin{align}\label{EQ.LEM.QOT_Distance_AC_L2_11}
\mathcal{W}_{\nabla_{\hspace{-0.055cm} j}}^{f,\theta}\lc\bar{\eta}_{j},\bar{\mu}_{j}^{\varepsilon}\rc\leq 2\sigma\lc\Delta_{j}\rc^{-\frac{1}{2}}C_{\delta}^{-\frac{\theta}{2}}\varepsilon^{-\frac{\theta}{2}}\dblv{}\sharp\bar{\mu}_{j}^{\varepsilon}-\sharp\bar{\eta}_{j}\dblv_{\tau}
\end{align}

\noindent for all $j\geq m$ in $\mathbb{N}$. Using triangle inequality, Equation \ref{EQ.LEM.QOT_Distance_AC_L2_10} and Equation \ref{EQ.LEM.QOT_Distance_AC_L2_11} show

\begin{align}\label{EQ.LEM.QOT_Distance_AC_L2_12}
\mathcal{W}_{\nabla_{\hspace{-0.055cm} j}}^{f,\theta}\lc\bar{\mu}_{j},\bar{\eta}_{j}\rc\leq 2\sigma\lc\Delta_{j}\rc^{-\frac{1}{2}}C_{\delta}^{-\frac{\theta}{2}}\varepsilon^{-\frac{\theta}{2}}\cdot \lc\dblv{}\varepsilon h^{\perp}\lc\sharp\bar{\mu}_{j}\rc\dblv_{\tau}+\dblv{}\sharp\bar{\mu}_{j}^{\varepsilon}-\sharp\bar{\eta}_{j}\dblv_{\tau}\rc{}
\end{align}

\noindent for all $j\geq m$ in $\mathbb{N}$. Note $2)$ in Theorem \ref{THM.QOT_Distance} shows

\begin{align}\label{EQ.LEM.QOT_Distance_AC_L2_13}
\mathcal{W}_{\nabla}^{f,\theta}\lc\bar{\mu}_{j},\bar{\eta}_{j}\rc{}=\mathcal{W}_{\nabla_{\hspace{-0.055cm} j}}^{f,\theta}\lc\bar{\mu}_{j},\bar{\eta}_{j}\rc{}
\end{align}

\noindent in each case as well. Applying Equation \ref{EQ.LEM.QOT_Distance_AC_L2_13} to Equation \ref{EQ.LEM.QOT_Distance_AC_L2_12} yields

\begin{align}\label{EQ.LEM.QOT_Distance_AC_L2_14}
\mathcal{W}_{\nabla}^{f,\theta}\lc\bar{\mu}_{j},\bar{\eta}_{j}\rc\leq 2\sigma\lc\Delta_{j}\rc^{-\frac{1}{2}}C_{\delta}^{-\frac{\theta}{2}}\varepsilon^{-\frac{\theta}{2}}\cdot \lc\dblv{}\varepsilon h^{\perp}\lc\sharp\bar{\mu}_{j}\rc\dblv_{\tau}+\dblv{}\sharp\bar{\mu}_{j}^{\varepsilon}-\sharp\bar{\eta}_{j}\dblv_{\tau}\rc{}
\end{align}

\noindent for all $j\geq m$ in $\mathbb{N}$.\par


\pagebreak


Finally, $4)$ in Proposition \ref{PRP.Wstar_Derivation_Compression_II} implies

\begin{align}\label{EQ.LEM.QOT_Distance_AC_L2_15}
0<\sigma(\Delta)\leq\inf_{j\in\mathbb{N}}\hspace{0.025cm} \sigma\lc\Delta_{j}\rc{}
\end{align}

\noindent by Proposition \ref{PRP.Spectral_Gap}. Applying Equation \ref{EQ.LEM.QOT_Distance_AC_L2_15} to Equation \ref{EQ.LEM.QOT_Distance_AC_L2_14} yields

\begin{align}\label{EQ.LEM.QOT_Distance_AC_L2_16}
\mathcal{W}_{\nabla}^{f,\theta}\lc\bar{\mu}_{j},\bar{\eta}_{j}\rc\leq 2\sigma(\Delta)^{-\frac{1}{2}}C_{\delta}^{-\frac{\theta}{2}}\varepsilon^{-\frac{\theta}{2}}\cdot \lc\dblv{}\varepsilon h^{\perp}\lc\sharp\bar{\mu}_{j}\rc\dblv_{\tau}+\dblv{}\sharp\bar{\mu}_{j}^{\varepsilon}-\sharp\bar{\eta}_{j}\dblv_{\tau}\rc{}
\end{align}

\noindent for all $j\geq m$ in $\mathbb{N}$.\par
We apply limit inferior to both sides in Equation \ref{EQ.LEM.QOT_Distance_AC_L2_16}. In addition, we use l.s.c.~in $w^{*}$-topology for the left-hand side and $I=\s$-$\lim_{j\in\mathbb{N}}\pi_{j}^{A}$ for the right-hand side to get its $\|.\|_{\tau}$-limit. Altogether, applying limit inferior to Equation \ref{EQ.LEM.QOT_Distance_AC_L2_16} lets us estimate

\begin{align}\label{EQ.LEM.QOT_Distance_AC_L2_17}
\mathcal{W}_{\nabla}^{f,\theta}(\mu,\eta)\leq 2\sigma(\Delta)^{-\frac{1}{2}}C_{\delta}^{-\frac{\theta}{2}}\varepsilon^{-\frac{\theta}{2}}\cdot \bigg(\dblv{}\varepsilon h^{\perp}\lc\sharp\mu\rc\dblv_{\tau}+\dblv{}\sharp\mu^{\varepsilon}-\sharp\eta\dblv_{\tau}\bigg).
\end{align}

\noindent Note $\dblv{}\mu^{\varepsilon}-\eta\dblv_{\tau}=\dblv{}(1-\varepsilon)h^{\perp}(\mu)-h^{\perp}(\eta)\dblv_{\tau}$ since $\mu,\eta\in\Fix_{A}(\xi)$. Equation \ref{EQ.LEM.QOT_Distance_AC_L2_17} shows

\begin{align}\label{EQ.LEM.QOT_Distance_AC_L2_18}
\mathcal{W}_{\nabla}^{f,\theta}(\mu,\eta)\leq 2\sigma(\Delta)^{-\frac{1}{2}}C_{\delta}^{-\frac{\theta}{2}}\varepsilon^{-\frac{\theta}{2}}\cdot \bigg(\dblv{}h^{\perp}\lc\sharp\mu\rc\dblv_{\tau}+\dblv{}(1-\varepsilon)h^{\perp}\lc\sharp\mu\rc{}-h^{\perp}\lc\sharp\eta\rc\dblv_{\tau}\bigg).
\end{align}

\noindent If Equation \ref{EQ.LEM.QOT_Distance_AC_L2_10} and Equation \ref{EQ.LEM.QOT_Distance_AC_L2_11} hold for all $j\geq m$ in $\mathbb{N}$, then Equation \ref{EQ.LEM.QOT_Distance_AC_L2_18} in turn holds for $0<\delta<\sigma(\xi)$ fixed but arbitrary. Letting $\delta\downarrow 0$ in Equation \ref{EQ.LEM.QOT_Distance_AC_L2_18} therefore yields Equation \ref{EQ.LEM.QOT_Distance_AC_L2_6}. The latter lets us apply the second reduction.\par
Assume $A$ and $B$ are finite-dimensional. Let $\mu,\eta\in\Fix_{A}(\xi)$. We show Equation \ref{EQ.LEM.QOT_Distance_AC_L2_10} and Equation \ref{EQ.LEM.QOT_Distance_AC_L2_11}. We suppress subscript $j\in\mathbb{N}$ without loss of generality. Set
 
\begin{align}\label{EQ.LEM.QOT_Distance_AC_L2_19}
\mu^{\varepsilon}(s):=\lc{}1-s\rc\mu+s\xi,\ \eta^{\varepsilon}(t):=\lc{}1-t\rc\eta+t\mu^{\varepsilon}
\end{align}

\noindent for all $s\in [0,\varepsilon]$ and $t\in [0,1]$. Note $\mu^{\varepsilon}:[0,\varepsilon]\longrightarrow\vartheta(\xi)$ and $\eta^{\varepsilon}:[0,1]\longrightarrow\vartheta(\xi)$ are absolutely continuous. Using the map $t\mapsto\varphi(t):=\varepsilon t$, rescaling $\mu:=\mu^{\varepsilon}\circ\varphi$ as per Remark \ref{REM.Energy_Functional_Reparametrisation} yields absolutely continuous $\mu:[0,\varepsilon]\longrightarrow\vartheta(\xi)$. We may use double notation for $\mu$ and $\mu^{\varepsilon}$, each denoting state and path, since their meaning is clear from context. We have $\dot{\mu}(t)=-\varepsilon h^{\perp}(\mu)$ and $\dot{\eta}^{\varepsilon}=\mu^{\varepsilon}-\eta$ in each case. Proposition \ref{PRP.RM_III} shows $\mu,\eta^{\varepsilon}:[0,1]\longrightarrow\vartheta(\xi)$ induce admissible paths

\begin{align}\label{EQ.LEM.QOT_Distance_AC_L2_20}
\lc\mu,\Theta\lc\mu,-\varepsilon h^{\perp}(\mu)\rc\rc\in\Admnullone\lc\mu,\mu^{\varepsilon}\rc{},\ \lc\eta^{\varepsilon},\Theta\lc\eta^{\varepsilon},\mu^{\varepsilon}-\eta\rc\rc\in\Admnullone\lc\eta,\mu^{\varepsilon}\rc{}.
\end{align}


\pagebreak


Since $\sharp\xi>0$ in $A_{\xi}$, we have

\begin{align}\label{EQ.LEM.QOT_Distance_AC_L2_21}
\sharp\mu(t),\sharp\eta^{\varepsilon}(t)\geq t\varepsilon\cdot \sharp\xi\geq t\varepsilon\cdot \sigma(\xi)\cdot \supp\xi
\end{align}

\noindent in $A_{\xi}$ for all $t\in (0,1]$. Using Lemma \ref{LEM.Spectral_Gap_Maximal}, Equation \ref{EQ.LEM.QOT_Distance_AC_L2_21} shows

\begin{align}\label{EQ.LEM.QOT_Distance_AC_L2_22}
\sigma\lc\mu(t)\rc{},\sigma\lc\eta^{\varepsilon}(t)\rc\geq t\varepsilon\cdot \sigma(\xi)
\end{align}

\noindent for all $t\in (0,1]$ by maximality of spectral gaps. Using $1)$ in Proposition \ref{PRP.RM_I}, then note Equation \ref{EQ.LEM.QOT_Distance_AC_L2_22} in turn shows

\begin{align}\label{EQ.LEM.QOT_Distance_AC_L2_23}
\mathfrak{F}_{\mu(t)}^{-1},\mathfrak{F}_{\eta^{\varepsilon}(t)}^{-1}\leq\sigma(\Delta)^{-1}t^{-\theta}\varepsilon^{-\theta}\sigma(\xi)^{-\theta}\cdot I
\end{align}

\noindent on $\im\Delta_{\xi}$. We evaluate paths on lengths functionals. Using $2)$ in Proposition \ref{PRP.RM_II} in order to evaluate on the Riemannian metric, Equation \ref{EQ.LEM.QOT_Distance_AC_L2_23} lets us estimate

\begin{align}\label{EQ.LEM.QOT_Distance_AC_L2_24}
L^{f,\theta}\lc\mu,\Theta\lc\mu,-\varepsilon h^{\perp}(\mu)\rc\rc\leq \sigma(\Delta)^{-\frac{1}{2}}\sigma(\xi)^{-\frac{\theta}{2}}\varepsilon^{-\frac{\theta}{2}}\cdot \dblv{}\varepsilon h^{\perp}\lc\sharp\mu\rc\dblv_{\tau}\cdot \int_{0}^{1}t^{-\frac{\theta}{2}}dt
\end{align}

\noindent and

\begin{align}\label{EQ.LEM.QOT_Distance_AC_L2_25}
L^{f,\theta}\lc\eta^{\varepsilon},\Theta\lc\eta^{\varepsilon},\mu^{\varepsilon}-\eta\rc\rc\leq\sigma(\Delta)^{-\frac{1}{2}}\sigma(\xi)^{-\frac{\theta}{2}}\varepsilon^{-\frac{\theta}{2}}\cdot \dblv{}\sharp\mu^{\varepsilon}-\sharp\eta\dblv_{\tau}\cdot \int_{0}^{1}t^{-\frac{\theta}{2}}dt.
\end{align}

\noindent Note $\int_{0}^{1}t^{-\frac{\theta}{2}}dt=\big(1-\frac{\theta}{2}\big)^{-1}\leq 2<\infty$ since $\theta\in [0,1]$. Following Corollary \ref{COR.Length_Functional_Reparametrisation}, we obtain the required estimates at once by minimising the left-hand sides in Equation \ref{EQ.LEM.QOT_Distance_AC_L2_24} and Equation \ref{EQ.LEM.QOT_Distance_AC_L2_25} over admissible paths with marginals chosen accordingly.
\end{proof}

The statement of Theorem \ref{THM.QOT_Distance_AC_L2}, resp.~Corollary \ref{COR.QOT_Distance_AC_L2}, refers to continuity defined on accessibility components of square integrable normal states by norm topology under the standard modified pairing.

\begin{thm}\label{THM.QOT_Distance_AC_L2}
Let $(\phi,\bpsi,\gamma,\nabla)$ be noncommutative differential structure for tracial AF-$C^{*}$-algebras $(A,\tau)$ and $(B,\omega)$ in $\lc{}f,\theta\rc$-setting. Assume $\sigma(\Delta)>0$. If $\xi\in\mathcal{S}_{-1}^{\NI,2}(A_{\xi})$ is a fixed state, then

\begin{itemize}
\item[1)] $\mathcal{C}_{A}^{\NI,2}(\xi):=\mathcal{C}_{A}(\xi)\cap\mathcal{S}^{\NI,2}(A)=\Fix_{A}(\xi)\cap\mathcal{S}^{\NI,2}(A)$,

\item[2)] $\mathcal{W}_{\nabla\vert\mathcal{C}_{A}^{2}(\xi)\times\mathcal{C}_{A}^{2}(\xi)}^{f,\theta}$ is finite and $\|.\|_{\tau}$-continuous.
\end{itemize}
\end{thm}
\begin{proof}
Let $\xi\in\mathcal{S}_{-1}^{\NI,2}(A_{\xi})$ be a fixed state. Using $1.2)$ in Proposition \ref{PRP.AF_Cstar_Trace_Dualisation_II} and $1.3)$ in Proposition \ref{PRP.Wstar_Derivation_QG_HSG_Fixed_Part_I}, $2)$ in Corollary \ref{COR.QOT_Distance_AC_I} implies

\begin{align}\label{EQ.THM.QOT_Distance_AC_L2_1}
\mathcal{C}_{A}(\xi)\cap\mathcal{S}^{\NI,2}(A)\subset\textrm{Fix}_{A}(\xi)\cap\mathcal{S}^{\NI,2}(A). 
\end{align}

\noindent Lemma \ref{LEM.QOT_Distance_AC_L2} shows $\mathcal{W}_{\nabla}^{f,\theta}$ is finite on $\Fix_{A}(\xi)\cap\mathcal{S}^{\NI,2}(A)$, i.e.~converse to Equation \ref{EQ.THM.QOT_Distance_AC_L2_1}. Get $1)$. Note $h^{\perp}\in\BII\lc{}L^{2}(A,\tau)\rc$ is a projection by $2.1)$ in Proposition \ref{PRP.Wstar_Derivation_QG_HSG_I}.\par


\pagebreak


We show $2)$. Let $\mu\in\mathcal{C}_{A}^{\NI,2}(\xi)$ and $\lset\mu^{n}\rset_{n\in\mathbb{N}}\subset\mathcal{C}_{A}^{\NI,2}(\xi)$ s.t.~$\sharp\mu=\|.\|_{\tau}$-$\lim_{n\in\mathbb{N}}\sharp\mu^{n}$. Using Lemma \ref{LEM.QOT_Distance_AC_L2}, there exists $C>0$ s.t.~

\begin{align}\label{EQ.THM.QOT_Distance_AC_L2_2}
0\leq\limsup_{n\in\mathbb{N}}\hspace{0.025cm} \mathcal{W}_{\nabla}^{f,\theta}\lc\mu,\mu^{n}\rc\leq C\varepsilon^{1-\frac{\theta}{2}}\cdot \|h^{\perp}\|_{\BII\lc{}L^{2}(A,\tau)\rc{}}\cdot \|\sharp\mu\|_{\tau}=C\varepsilon^{1-\frac{\theta}{2}}\cdot \|\sharp\mu\|_{\tau}
\end{align}

\noindent for all $\varepsilon\in (0,1]$. Letting $\varepsilon\downarrow 0$ in Equation \ref{EQ.THM.QOT_Distance_AC_L2_2} shows $\|.\|_{\tau}$-continuity by l.s.c.~as per $3)$ in Theorem \ref{THM.QOT_Distance}. Get $2)$.
\end{proof}

\begin{cor}\label{COR.QOT_Distance_AC_L2}
Assume $A$ and $B$ are finite-dimensional. If $\xi\in\SII(A)=\mathcal{S}^{\NI}(A)$ is a fixed state, then

\begin{itemize}
\item[1)] $\mathcal{C}_{A}^{\NI,2}(\xi)=\mathcal{C}_{A}(\xi)=\overline{\vartheta(\xi)}^{\|.\|_{A}}=\Fix_{A}(\xi)=\Fix_{A}^{\NI}(\xi)$,

\item[2)] $\mathcal{W}_{\nabla\vert\mathcal{C}_{A}(\xi)\times\mathcal{C}_{A}(\xi)}^{f,\theta}$ is finite and $\|.\|_{A}$-continuous.
\end{itemize} 
\end{cor}
\begin{proof}
Get $\overline{\vartheta(\xi)}^{\|.\|_{A}}=\Fix_{A}(\xi)$ since $\vartheta(\xi)=\relint\Fix_{A}(\xi)$ by $1)$ in Proposition \ref{PRP.RM_Embedded_Submanifold_II}. We see Theorem \ref{THM.QOT_Distance_AC_L2} implies $1)$ and $2)$ as $\xi\in\mathcal{S}_{-1}^{\NI,2}(A_{\xi})$ by finite-dimensionality.
\end{proof}

\begin{bsp}\label{BSP.QOT_Distance_AC_L2_2}
Assume the setting of Example \ref{BSP.QOT_Type_II_1}. Let $\{e_{j}\}_{j\in\mathbb{N}}$ be orthonormal eigenbasis of $D\in\UBII(H)_{h}$. For all $j\in\mathbb{N}$, get $H_{j}=\langle e_{1},\ldots,e_{j}\rangle_{\mathbb{C}}$ and $A_{j}=\AII(H_{j}[J])$. Note $(A_{j},\tau)\cong \lc\otimes_{k=1}^{j}M_{2}(\mathbb{C}),2^{-j}\otimes_{k=1}^{j}\tr_{2}\rc\cong \lc{}M_{2^{j}}(\mathbb{C}),2^{-j}\tr_{2^{j}}\rc$ in each case \lc{}cf.~p.288 in \cite{BK.Nes_Sto.2006.Rel_Ent}\rc{}.\par
We give the $C^{*}$-isomorphism. Let $j\in\mathbb{N}$. Set 

\begin{align}\label{EQ.BSP.QOT_Distance_AC_L2_1}
V_{J}(e_{j}):=a_{J}(e_{j})a_{J}(e_{j})^{*}-a_{J}(e_{j})^{*}a_{J}(e_{j}).
\end{align}

\noindent For all $k\in\lset{}1,\ldots,j\rset$ and $n,m\in\lset{}1,2\rset$, let $E_{nm}^{k}$ be the $\lc{}n,m\rc$-unit matrix of $M_{2}(\mathbb{C})$ in the $k$-th factor of $\otimes_{k=1}^{j}M_{2}(\mathbb{C})$. We define $C^{*}$-isomorphism from $\otimes_{k=1}^{j}M_{2}(\mathbb{C})$ to $A_{j}$ by setting

\begin{align*}
E_{nm}^{k}\cong
\begin{cases} 
a_{J}(e_{k})a_{J}(e_{k})^{*} & \If\ n=m=1, \\
a_{J}(e_{k})^{*}a_{J}(e_{k}) & \If\ n=m=2, \\
a_{J}(e_{k})V_{J}\lc{}e_{1}\rc\cdots V_{J}\lc{}e_{k-1}\rc{} & \If\ n=1,\ m=2, \\
a_{J}(e_{k})^{*}V_{J}\lc{}e_{1}\rc\cdots V_{J}\lc{}e_{k-1}\rc{} & \If\ n=2,\ m=1.
\end{cases}
\end{align*}

\noindent Letting $j\uparrow\infty$ provides orthonormal basis of $L^{2}\lc\mathcal{A}(H),\tau\rc$ as follows, moreover suited to calculate sufficient conditions for quantum Laplacian with spectral gaps. Indexed over $k\in\mathbb{N}$ and $n,m\in\lset{}1,2\rset$, set

\begin{align*}
\mathbf{E}_{nm}^{k}:=
\begin{cases} 
E_{11}^{k}-E_{22}^{k} & \If\ n=m, \\
2E_{nm}^{k} & \If\ n\neq m.
\end{cases}
\end{align*}


\pagebreak


The set of all finite products $\mathbf{E}$ using factors in $\lset{}I,\mathbf{E}_{nm}^{k}\rset_{k\in\mathbb{N},n,m\in\lset{}1,2\rset}$ is orthonormal basis of $L^{2}\lc\mathcal{A}(H),\tau\rc$. For all $j\in\mathbb{N}$, let $\nu_{j}$ be the eigenvalue of $e_{j}$ and use Equation \ref{EQ.BSP.QOT_Type_II_6} in order to calculate

\begin{align}\label{EQ.BSP.QOT_Distance_AC_L2_2}
\nabla a_{J}(e_{j})=\restr{0.925}{\frac{d}{dt}}{t=0,\w}\alpha_{t}\lc{}a_{J}(e_{j})\rc{}=\restr{0.925}{\frac{d}{dt}}{t=0,\w}e^{-it\nu_{j}}\cdot a_{J}(e_{j})=-i\nu_{j}\cdot a_{J}(e_{j})
\end{align}

\noindent and therefore

\begin{align}\label{EQ.BSP.QOT_Distance_AC_L2_3}
\nabla a_{J}(e_{j})^{*}=\lc\nabla a_{J}(e_{j})\rc^{*}=i\nu_{j}\cdot a_{J}(e_{j}).
\end{align}

\noindent Equation \ref{EQ.BSP.QOT_Distance_AC_L2_2} and Equation \ref{EQ.BSP.QOT_Distance_AC_L2_3} show $\nabla V_{J}(e_{j})=0$ in each case. If $n=m$, then we further have $V_{J}(e_{j})=\mathbf{E}_{nm}^{j}$. For all $\mathbf{E}_{nm}^{k}\in\mathbf{E}$, the Leibniz rule therefore implies

\begin{align}\label{EQ.BSP.QOT_Distance_AC_L2_4}
\nabla\mathbf{E}_{nm}^{k}=\sum_{j\in I}\nu_{j}\cdot\mathbf{E}_{nm}^{k}    
\end{align}

\noindent for finite $I\subset\mathbb{N}$ depending on $\mathbf{E}_{nm}^{k}$. Since $\nabla^{*}=-\nabla$, we see Equation \ref{EQ.BSP.QOT_Distance_AC_L2_4} consequently shows all eigenvalues $\lambda$ of $\Delta$ have form

\begin{align}\label{EQ.BSP.QOT_Distance_AC_L2_5}
\lambda=\bbbabsv{0.95}{\hspace{0.0225cm}\sum_{j\in I}\nu_{j}\hspace{0.0675cm}}^{2}.
\end{align}

\noindent Assume there exists $C>0$ s.t.~$\nu_{j}\in C\mathbb{Z}$ for all $j\in\mathbb{N}$. Then Equation \ref{EQ.BSP.QOT_Distance_AC_L2_5} shows $\lambda=C^{2}q^{2}$ for $q\in\mathbb{Z}$ in each case. We have $\lambda=0$ if and only if $q=0$. Thus $\lambda\neq 0$ implies $|q|\geq 1$ and therefore $\lambda\geq C^{2}>0$, hence $\Delta$ either vanishes or has spectral gap.
\end{bsp}


\section{Coarse graining and transport of quantum information}\label{SEC.QOT_QIT}

We consider states on tracial AF-${C}^{*}$-algebras as scaling limits of uniformly conditioned spin states encoding sequences of qubits. Scaling limits arise from a coarse graining process associated to noncommutative differential structures. We view quantum optimal transport as transport of quantum information. Since energy functionals are $\Gamma$-limits w.r.t.~the coarse graining process, resp.~using our formalisation of the latter notion in Subsection \ref{SSEC.QOT_DT_MG}, we view minimising geodesics approximated in finite dimensions as optimal transport of information encoded in scaling limits as above.

\medskip

\noindent\textbf{Structure.} In Subsection \ref{SSEC.QOT_QIT_Encoding}, we discuss coarse graining and scaling limits. We consider states on tracial AF-${C}^{*}$-algebras as scaling limits of uniformly conditioned spin states encoding sequences of qubits. In Subsection \ref{SSEC.QOT_CG}, we give the coarse graining process and view quantum optimal transport as transport of quantum information.


\subsection[Information encoded in states on tracial AF-$C^{*}$-algebras]{Information encoded in states on tracial AF-$\mathbf{C}^{*}$-algebras}\label{SSEC.QOT_QIT_Encoding}

The fundamental unit of quantum information is the quantum bit, or qubit \cite{BK.Nie_Chu.2000.Quantum_Computation_Information}\cite{ART.DiVi_Loss.1998.Quantum_Information_Physical}. We consider spin states encoding qubits \cite{ART.Bur_DiVi_Loss.1999.QC_Spin_QDots_as_QGates} since spin qubit quantum computers \cite{ART.Bur_Lad_Nic.2023.QC_Spin_Overview}\cite{BK.Nie_Chu.2000.Quantum_Computation_Information} operationalise \cite{ART.Ash_Geo_Nor.2014.Quantum_Simulation} spin states according to DiVicenzo's criteria \cite{ART.DiVi.2000.Criteria}\cite{ART.DiVi_Loss.1998.Quantum_Information_Physical}. We generalise to scaling limits of uniformly conditioned spin states encoding sequences of qubits. We show states on tracial AF-$C^{*}$-algebras encode information in such form.\par
We do not claim they have physical realisation in general. However, we show such states are noncommutative analogues of scaling limits arising from projective limits of Banach dual spaces. These are themselves dualisations of direct limits in the category of commutative $C^{*}$-algebras obtained by means of a coarse graining process for locally compact Hausdorff spaces. Spin states are a special case and have well-known physical realisations as spin qubits \cite{ART.Bur_DiVi_Loss.1999.QC_Spin_QDots_as_QGates}\cite{ART.Bur_Lad_Nic.2023.QC_Spin_Overview}\cite{BK.Nie_Chu.2000.Quantum_Computation_Information}\cite{ART.DiVi.2000.Criteria}\cite{ART.DiVi_Loss.1998.Quantum_Information_Physical}. Standard reference for approaches and methods of coarse graining in the commutative setting is \cite{BK.Gor_Kaz_Ott_The.2006.Coarse_Graining}. Standard reference for category theory is \cite{BK.MacLane.1998.Category_Theory}. Standard reference for quantum information theory is \cite{BK.Nie_Chu.2000.Quantum_Computation_Information}.


\subsubsection*{Coarse graining and scaling limits}

Following a general description of coarse graining via renormalisation group transformations \lc{}cf.~pp.180-182 in \cite{BK.Gor_Kaz_Ott_The.2006.Coarse_Graining}\rc{}, we obtain scaling limits from direct limits in the category of commutative $C^{*}$-algebras by means of a coarse graining process for locally compact Hausdorff spaces. Dualisation furthermore yields projective limits of their Banach dual spaces. Examples arise from Ehrenfest coarse graining processes for continuity equations \lc{}cf.~pp.117-140 in \cite{BK.Gor_Kaz_Ott_The.2006.Coarse_Graining}\rc{}. We show the AF-$C^{*}$-setting yields noncommutative analogues of scaling limits.\par
We review the classical case. We use Gelfand-Naimark functor defined by Gelfand duality \lc{}cf.~Theorem I.3.11 in \cite{BK.Tak.1979.OpAlg_I}\rc{}. It yields natural transformation for the categories of locally compact Hausdorff spaces and commutative $C^{*}$-algebras \lc{}cf.~Theorem I.4.4 in \cite{BK.Tak.1979.OpAlg_I}\rc{}. The classical case is in said commutative setting \lc{}cf.~Example \ref{BSP.Cstar_Commutative}\rc{}. We use direct and projective limits \lc{}cf.~pp.62-72 in \cite{BK.MacLane.1998.Category_Theory}\rc{}. Let $X$ be a locally compact Hausdorff space. We view $X$ as phase space of a physical system \cite{BK.Gor_Kaz_Ott_The.2006.Coarse_Graining}\cite{BK.Ste_vLee.2013.Full_Quantum_StM}. Let $\lset{}X_{j}\rset_{j\in\mathbb{N}}$ be locally compact Hausdorff spaces s.t.~we have diagram of continuous surjective maps

\begin{equation}\label{EQ.SSEC.QOT_QIT_Encoding_1}
\begin{tikzcd}
X\arrow[r, two heads] & \ \cdots \arrow[r, two heads] & X_{j}\arrow[r, two heads] & \cdots \arrow[r, two heads] & X_{1}
\end{tikzcd}
\end{equation}

\noindent in the category of locally compact Hausdorff spaces. Assume Diagram \ref{EQ.SSEC.QOT_QIT_Encoding_1} maps, under the Gelfand-Naimark functor, to the direct limit diagram

\begin{equation}\label{EQ.SSEC.QOT_QIT_Encoding_2}
\begin{tikzcd}
C_{0}(X_{1})\arrow[r, hook] & \ \cdots \arrow[r, hook] & C_{0}(X_{j})\arrow[r, hook] & \cdots \arrow[r, hook] & C_{0}(X)=\varinjlim C_{0}(X_{j})
\end{tikzcd}
\end{equation}

\noindent in the category of commutative $C^{*}$-algebras. Dualisation reverts arrows and therefore maps Diagram \ref{EQ.SSEC.QOT_QIT_Encoding_2} to the projective limit diagram

\begin{equation}\label{EQ.SSEC.QOT_QIT_Encoding_3}
\begin{tikzcd}
C_{0}(X)^{*}=\varprojlim C_{0}(X_{j})^{*} \arrow[r, two heads] & \ \cdots \arrow[r, two heads] & C_{0}(X_{j})^{*}\arrow[r, two heads] & \cdots \arrow[r, two heads] & C_{0}(X_{1})^{*}
\end{tikzcd}
\end{equation}

\noindent in the category of Banach spaces.\par


\pagebreak


The set of pure states on $C_{0}(X)$ is the set of Dirac measures on $X$ \lc{}cf.~Theorem I.3.11 and Definition I.3.12 in \cite{BK.Tak.1979.OpAlg_I}\rc{}. We view the latter as pointwise measurement of phase space $X$. If each $X_{j}=X\big/\hspace{-0.085cm}\sim_{j}$ is a quotient space for a directed set $\lset\sim_{j}\rset_{j\in\mathbb{N}}$ of equivalence relations on $X$ in dual order, then each step in Diagram \ref{EQ.SSEC.QOT_QIT_Encoding_1} identifies certain sets of pointwise measurements. We thereby define renormalisation group transformations and obtain a coarse graining process \lc{}cf.~p.181 in \cite{BK.Gor_Kaz_Ott_The.2006.Coarse_Graining}\rc{}. Examples arise from identifying interiors of certain cells in Ehrenfest coarse graining \lc{}cf.~pp.117-123 in \cite{BK.Gor_Kaz_Ott_The.2006.Coarse_Graining}\rc{}.\par
We see injections in Diagram \ref{EQ.SSEC.QOT_QIT_Encoding_2} are inclusions of observables on phase space $X$ invariant under certain pointwise measurements. Each step in Diagram \ref{EQ.SSEC.QOT_QIT_Encoding_2} increases the set of observables s.t.~more pointwise measurements are separated. Diagram \ref{EQ.SSEC.QOT_QIT_Encoding_3} further extends Diagram \ref{EQ.SSEC.QOT_QIT_Encoding_1} by extending it to all totally finite signed outer regular Radon measures on $X$ \lc{}cf.~Theorem 6.3.4 in \cite{BK.Ped.1989.Analysis_Now}\rc{} with separability increasing in each step. Sets of identified, i.e.~non-separated, pointwise measurements have characteristic scale, e.g.~the volume of cell interiors. Letting $j\uparrow\infty$ implies these tend to zero since

\begin{align}\label{EQ.SSEC.QOT_QIT_Encoding_4}
C_{0}(X)=\overline{\bigcup_{j\in\mathbb{N}}C_{0}(X_{j})}^{\|.\|_{\infty}}.
\end{align}

\noindent We say that elements in $C_{0}(X)$ and $C_{0}(X)^{*}$, as well as compatible objects or properties using the latter, are scaling limits.\par
We show the AF-$C^{*}$-setting yields noncommutative analogues of scaling limits. Let $(A,\tau)$ be a tracial AF-$C^{*}$-algebra. Definition \ref{DFN.AF_Cstar} shows direct limit diagram

\begin{equation}\label{EQ.SSEC.QOT_QIT_Encoding_5}
\begin{tikzcd}
A_{1}\arrow[r, hook] & \ \cdots \arrow[r, hook] & A_{j}\arrow[r, hook] & \cdots \arrow[r, hook] & A=\overline{A_{0}}^{\|.\|_{A}}
\end{tikzcd}
\end{equation}

\noindent in the category of $C^{*}$-algebras. For all $j\in\mathbb{N}$, $2)$ in Proposition \ref{PRP.AF_Cstar_Trace_Dualisation_I} shows $A_{j}^{*}\subset A^{*}$. We use the modified standard pairing. Dualisation maps Diagram \ref{EQ.SSEC.QOT_QIT_Encoding_5} to the projective limit diagram

\begin{equation}\label{EQ.SSEC.QOT_QIT_Encoding_6}
\begin{tikzcd}
A^{*}=\overline{A_{0}^{*}}^{\|.\|_{A^{*}}}\arrow[r, two heads] & \ \cdots \arrow[r, two heads] & A_{j}^{*}\arrow[r, two heads] & \cdots \arrow[r, two heads] & A_{1}^{*}
\end{tikzcd}
\end{equation}

\noindent in the category of Banach spaces. Following our discussion immediately above, we see elements in $A$ and $A^{*}$, as well as compatible objects or properties using the latter, are noncommutative analogues of scaling limits. Diagram \ref{EQ.SSEC.QOT_QIT_Encoding_6} gives the coarse graining process without rescaling or consideration for the metric geometry of quantum optimal transport distances. In Subsection \ref{SSEC.QOT_CG}, Diagram \ref{EQ.SSEC.QOT_QIT_Encoding_16} extends Diagram \ref{EQ.SSEC.QOT_QIT_Encoding_6}.\par
We are motivated by Ehrenfest coarse graining since it provides a coarse graining process lifting kinetic equations on phase spaces to continuity equations on state spaces by cell averaging \lc{}cf.~pp.123-129 in \cite{BK.Gor_Kaz_Ott_The.2006.Coarse_Graining}\rc{}. However, we neither coarse grain time nor use entropy production to control scaling limits. As such, we do not see Diagram \ref{EQ.SSEC.QOT_QIT_Encoding_16} to be a noncommutative analogue of Ehrenfest coarse graining. The maximum entropy production principle given in Subsection \ref{SSEC.L2W_Log_Mean_QNE} is, to our knowledge, unrelated.

\subsubsection*{Spin states}

We view tracial $C^{*}$-algebras as algebras of observables \cite{BK.Dav.1976.Quantum_Markov_SG}\cite{ART.Dav_Lew.1970.Wstar_Quantum_Probability}\cite{BK.Gar_Zol.2004.Quantum_Noise}\linebreak\cite{BK.Ohy_Pet.1993.Rel_Ent}\cite{BK.Ste_vLee.2013.Full_Quantum_StM}\cite{BK.Tak.1979.OpAlg_I} used in Hamiltonian formalism \cite{BK.Bra.1987.OpAlg_Quantum_StM_I}\cite{BK.Bra.1987.OpAlg_Quantum_StM_II}\cite{BK.Dav.1976.Quantum_Markov_SG}\cite{BK.Gar_Zol.2004.Quantum_Noise}\cite{BK.Ohy_Pet.1993.Rel_Ent}\cite{BK.Ste_vLee.2013.Full_Quantum_StM} for a given quantum system. The set of all propositions $P$ on a given quantum system is a lattice of projections \lc{}cf.~pp.1-11 in \cite{BK.Ohy_Pet.1993.Rel_Ent}\rc{}. If $P$ is equipped with f.s.n.~weight $\omega:P\longrightarrow [0,\infty]$ \cite{BK.Tak.2003.OpAlg_II}, then the GNS-construction for weights defines a faithful unital $^{*}$-representation. This yields generated $W^{*}$-algebra $W^{*}\lc{}P\rc$ \lc{}cf.~Proposition \ref{PRP.Wstar_Equivalence} and Definition \ref{DFN.Wstar_Generated}\rc{}. All tracial $W^{*}$-algebras arise in this manner \lc{}cf.~Proposition \ref{PRP.Wstar_Generated}\rc{}. Let $(A,\tau)$ be a tracial $C^{*}$-algebra. We have f.s.n.~trace $\tau:=\omega:P\lc{}L^{\infty}(A,\tau)\rc\longrightarrow [0,\infty]$ and $L^{\infty}(A,\tau)=W^{*}\lc{}P\lc{}L^{\infty}(A,\tau)\rc\rc$ \lc{}cf.~Proposition \ref{PRP.Wstar_Generated}\rc{}. It suffices to consider $A\subset L^{\infty}(A,\tau)$ as algebra of observables since it is a $\sigma$-weakly dense $C^{*}$-subalgebra. Altogether, we view $A$ as algebra of observables for the quantum system described by the set of all propositions $P\lc{}L^{\infty}(A,\tau)\rc$. We view $\SII(A)$ as its set of states \cite{BK.Ohy_Pet.1993.Rel_Ent}\cite{BK.Tak.1979.OpAlg_I}. Following Remark \ref{REM.Wstar_CP_Markovian_SG}, we know precomposition with quantum channels transmits change of such states.\par
We consider spin states encoding qubits under quantum noise. We do not specify the latter here. However, Example \ref{BSP.L2W_Log_Mean_QNE_QG_Internal} gives the depolarising channel as canonical choice of quantum noise operator \lc{}cf.~pp.378-379 in \cite{BK.Nie_Chu.2000.Quantum_Computation_Information}\rc{}. Let $n\in\mathbb{N}$. Up to scaling of density operators \lc{}cf.~pp.98-105 in \cite{BK.Nie_Chu.2000.Quantum_Computation_Information}\rc{}, pure states of $n$ qubits are given by all Hilbert space projections onto one-dimensional subspaces of $H:=\otimes_{k=1}^{n}\mathbb{C}^{2}$ \lc{}cf.~pp.13-17 in \cite{BK.Nie_Chu.2000.Quantum_Computation_Information}\rc{}. They generate, by construction as a subset of all propositions on a given quantum system with state vectors in $H$, the lattice $P$ of all Hilbert space projections onto any subspace of $H$. Assume $(A,\tau)=\lc\otimes_{k=1}^{n}M_{2}(\mathbb{C}),2^{-n}\otimes_{k=1}^{n}\tr_{2}\rc$. Thus $A=W^{*}\lc{}P\rc$, hence $A$ is an algebra of observables as above. Corollary \ref{COR.Support_Projection_II} implies pure states on $A$, i.e.~the extreme points of $\SII(A)$, are pure states of $n$ qubits. Superposition shows $\SII(A)$ are states of $n$ qubits. Spin qubit quantum computer \cite{ART.Bul_Bur_Loss_Trau.2007.QC_Spin_QDots_in_Graphene}\cite{ART.Bur_DiVi_Loss.1999.QC_Spin_QDots_as_QGates}\cite{ART.Bur_Lad_Nic.2023.QC_Spin_Overview}\cite{ART.DiVi_Loss.1998.QC_Spin_QDots} use spin-entangled electrons \cite{ART.Bur.2007.QC_Spin_Entanglement}\cite{ART.Bur_Lad_Nic.2023.QC_Spin_Overview} as physical realisation of $\SII(A)$ in order to achieve scalable quantum computing according to DiVicenzo's criteria \cite{ART.DiVi.2000.Criteria}\cite{ART.DiVi_Loss.1998.Quantum_Information_Physical}. If initialisation prepares pure states and quantum gates are unitary operations, then $1)$ in Corollary \ref{COR.Support_Projection_III} implies quantum computations are restricted to $\partial\SII(A)$. This is a desired feature but does require challenging control of quantum noise in form of sufficient quantum error correction \cite{ART.Bur_Lad_Nic.2023.QC_Spin_Overview}\cite{BK.Nie_Chu.2000.Quantum_Computation_Information}. The latter may be relaxed to initialisation preparing mixed states while retaining an edge over classical computing \cite{ART.Fri_Sio.2011.QC_Spin_Mixed_States}. We consider each $\mu\in\SII(A)$ as spin state of $n$ qubits under quantum noise and say that it encodes the latter. We ignore the r\^ole of quantum noise here.\par
Spin is an intrinsic property of elementary particles, e.g.~electrons, in the Standard Model of particle physics \cite{ART.Cha_Con_Mar.2007.NCG_Standard_Model_Recovered}\cite{ART.Gai_Gra_Sci.1999.The_Standard_Model}\cite{BK.vSui.2015.NCG_AF_Particle_Physics}. Its independence from mass, in contrast to angular momentum, necessitates use of spinors \cite{BK.Ply_Rob.1994.Clifford_Algebras}\cite{BK.vSui.2015.NCG_AF_Particle_Physics}\cite{BK.Var.2006.NCG_Elements_Short} in the Dirac equation \cite{BK.Tha.1992.The_Dirac_Equation}. Together with non-spatiality as per Example \ref{BSP.QOT_First_Quantisation} and Example \ref{BSP.QOT_Second_Quantisation}, this motivates our view of quantum optimal transport as non-spatial transport of quantum information. If we obtain the latter as analogue quantum simulation \cite{ART.Ash_Geo_Nor.2014.Quantum_Simulation} for sufficiently small $n\in\mathbb{N}$, then we have physical realisation of our interpretation. Noncommutative analogues of push-forward measure representations \cite{ART.CorEra_McCan_Sch.2001.Displacement_Convexity_Riemannian}\cite{ART.McCan.1997.Displacement_Convexity_Local} given by precomposition with quantum channels as per Remark \ref{REM.L2W_EVI_Ric} provide an ansatz but are not known to exist. If we further obtain the classical case as analogue simulation \cite{COL.Bou_Pou.Survey_Analog_Models_Computation}\cite{COL.MacLen.2014.Analog_Computation}, e.g.~for fluid dynamics \cite{ART.Ben_Bre.2000.Dynamic_OT}\cite{ART.Dol_Naz_Sav.2009.Generalised_OT} but without any spatial discretisation of observables, then we suspect similarities and differences of either arise from distinct physical realisations.


\subsubsection*{Scaling limits of uniformly conditioned spin states}

Note all formulations of the classical case implicitly assume pure states have vanishing support, i.e.~are Dirac measures. Assuming non-atomic Radon measure, Dirac delta sequences \cite{BK.Eva.2010.Partial_Differential_Equations}\cite{BK.Koe.1993.Analysis_II}\cite{BK.Koe.2004.Analysis_I} show infinitesimal length elements \cite{BK.Con.1994.NCG}\cite{BK.Lan.1995.Riemannian_Manifolds} imply all pure states have infinite relative entropy w.r.t.~the given Radon measure. We consider a different but equally well-known idealisation by letting $n\in\mathbb{N}$ tend to infinity. We thereby allow countable infinitely many interacting quantum systems, e.g.~second quantisation as per Example \ref{BSP.QOT_Second_Quantisation}, as initial approximation for a finite but large number of interacting ones \lc{}cf.~pp.3-5 in \cite{BK.Bra.1987.OpAlg_Quantum_StM_II}\rc{}. In Chapter \ref{CH.L2W}, we rectify the latter for our main contributions by restricting to the domain of quantum relative entropy. We therefore generalise spin states encoding qubits to scaling limits of uniformly conditioned spin states encoding sequences of qubits.\par
We show states on tracial AF-$C^{*}$-algebras are of such form, i.e.~we consider scaling limits of uniformly conditioned spin states encoding sequences of qubits. Let $(A,\tau)$ be a tracial AF-$C^{*}$-algebra. Remark \ref{REM.AF_Cstar_Trace_Dualisation_Admissible_Paths} explains use of restrictions in Equation \ref{EQ.SSEC.QOT_QIT_Encoding_7} below. For all $\mu\in\SII(A)$, we have

\begin{align}\label{EQ.SSEC.QOT_QIT_Encoding_7}
\mu=w^{*}\textrm{-}\lim_{j\in\mathbb{N}}\hspace{0.025cm} \mu_{j}=w^{*}\textrm{-}\lim_{j\in\mathbb{N}}\hspace{0.025cm} \bar{\mu}_{j}.
\end{align}

\noindent Following Diagram \ref{EQ.SSEC.QOT_QIT_Encoding_6}, note Equation \ref{EQ.SSEC.QOT_QIT_Encoding_7} lets us consider each $\mu\in\SII(A)$ as scaling limit. We rescale in each step for a given state but not uniformly on sets of states. We do so for Diagram \ref{EQ.SSEC.QOT_QIT_Encoding_16}. Here, we show how to consider a.e.~$\bar{\mu}_{j}\in A_{j,+}^{*}$ in Equation \ref{EQ.SSEC.QOT_QIT_Encoding_7} as uniformly conditioned spin state encoding qubits. We therefore consider each $\mu\in\SII(A)$ as scaling limit of uniformly conditioned spin states encoding a sequence of qubits.\par
We consider uniformly conditioned spin states encoding qubits. For all $n\in\mathbb{N}$, note Example \ref{BSP.QOT_Distance_AC_L2_2} gives an isomorphism $\lc\otimes_{k=1}^{n}M_{2}(\mathbb{C}),2^{-n}\otimes_{k=1}^{n}\tr_{2}\rc\cong \lc{}M_{2^{n}}(\mathbb{C}),2^{-n}\tr_{2^{n}}\rc$ of tracial $C^{*}$-algebras \cite{BK.Nes_Sto.2006.Rel_Ent}. Let $j\in\mathbb{N}$. There exists minimal $q_{j}\in\mathbb{N}$ s.t.~

\begin{align}\label{EQ.SSEC.QOT_QIT_Encoding_8}
A_{j}\overset{r_{A_{j}}}{\cong}\oplus_{l=1}^{n_{j}}M_{n_{j,l}}(\mathbb{C})\subset M_{2^{q_{j}}}(\mathbb{C})\cong\otimes_{k=1}^{q_{j}}M_{2}(\mathbb{C})
\end{align}

\noindent using inclusion $\oplus_{l=1}^{n_{j}}M_{n_{j,l}}(\mathbb{C})\subset M_{2^{q_{j}}}(\mathbb{C})$ into the upper left corner. Equation \ref{EQ.SSEC.QOT_QIT_Encoding_8} uses Notation \ref{NTN.AF_Cstar_Fin_Isometry}. Set

\begin{align}\label{EQ.SSEC.QOT_QIT_Encoding_9}
N:=\oplus_{l=1}^{n_{j}}M_{n_{j,l}}(\mathbb{C}),\ M:=\otimes_{k=1}^{q_{j}}M_{2}(\mathbb{C}).
\end{align}

\noindent We suppress the second $C^{*}$-isomorphism in Equation \ref{EQ.SSEC.QOT_QIT_Encoding_8} and consider $C^{*}$-subalgebra $N\subset M$. Using the latter, $1)$ in Proposition \ref{PRP.Wstar_Trace_NCE_II} yields noncommutative conditional expectation

\begin{align}\label{EQ.SSEC.QOT_QIT_Encoding_10}
\pi_{j}^{\textrm{sp}}:=\pi_{N}^{M}:M\longrightarrow N
\end{align}

\noindent from $M$ to $N$. Note $\pi_{j}^{\textrm{sp}}$ is unital, surjective and positivity-preserving. Moreover, we know it conditions the set of all propositions $P(M)$ on the given quantum system to a subset of propositions $P(N)$ \cite{BK.Tak.1979.OpAlg_I}.\par


\pagebreak


We obtain positivity-preserving injective Banach dual

\begin{align}\label{EQ.SSEC.QOT_QIT_Encoding_11}
\pi_{j}^{\textrm{sp},*}:=\lc\pi_{N}^{M}\rc^{*}:N^{*}\longrightarrow M^{*}
\end{align}

\noindent s.t.~$\pi_{j}^{\textrm{sp},*}\lc\SII(N)\rc\subset\SII(M)$. Precomposing with $\pi_{j}^{\textrm{sp}}$ restricts each $\mu\in\SII(N)$ from $M$ to $N$ by conditioning $P(M)$ to $P(N)$. We consider each $\mu\in\SII(N)$ as uniformly conditioned spin state of $q_{j}$ qubits and say that it encodes the latter. The first identity in Equation \ref{EQ.SSEC.QOT_QIT_Encoding_8} and Equation \ref{EQ.SSEC.QOT_QIT_Encoding_11} show we have positivity-preserving injective Banach dual

\begin{align}\label{EQ.SSEC.QOT_QIT_Encoding_12}
\iota_{j}^{\textrm{sp}}:=\lc{}r_{A_{j}}\circ\pi_{j}^{\textrm{sp}}\rc^{*}:A_{j}^{*}\longrightarrow M^{*}
\end{align}

\noindent s.t.~$\iota_{j}^{\textrm{sp}}\lc\SII(A_{j})\rc{}=\pi_{j}^{\textrm{sp},*}\lc\SII(N)\rc\subset\SII(M)$. Precomposing with $r_{A_{j}}$ transforms the set of all propositions from $P(N)$ to $P(A_{j})$ by equivalent formulation of observables. We consider each $\mu\in\SII(A_{j})$ as uniformly conditioned spin state of $q_{j}$ qubits and say that it encodes the latter. We furthermore consider scaling limits of uniformly conditioned spin states encoding a sequence of qubits as discussed above.


\subsection{Transport of quantum information}\label{SSEC.QOT_CG}

We give the coarse graining process and view quantum optimal transport as transport of quantum information. The coarse graining process involves rescaling and considers the metric geometry of quantum optimal transport distances. We use compression and finite-dimensional approximation as used for classification of accessibility components in Subsection \ref{SSEC.QOT_AC_RM} for its construction. We thereby formalise compatibility with both in the coarse graining process as claimed in Subsection \ref{SSEC.NCDS_NCG_Notion}.\par
The coarse graining process applies to accessibility components. These have unique common fixed parts ensuring existence of scaled restriction maps. In order to respect scaling limit description of marginals and fixed parts as per Subsection \ref{SSEC.QOT_QIT_Encoding}, we only consider minimising geodesics approximated in finite dimensions as optimal transport of scaling limits of of quantum information, i.e.~of uniformly conditioned spin states encoding sequences of qubits. Non-ergodicity restricts information-bearing degrees of freedom by the continuity equation. Moreover, the coarse graining process reduces the AF-$C^{*}$-setting to the finite-dimensional one s.t.~ergodicity is recovered up to fixed parts by reducing to those accessibility components in the finite-dimensional setting arising from scaled restriction of the given fixed part. For this, we use classification to determine accessibility components in the finite-dimensional setting.


\subsubsection*{The coarse graining process} 

Diagram \ref{EQ.SSEC.QOT_QIT_Encoding_16} extends Diagram \ref{EQ.SSEC.QOT_QIT_Encoding_6} and gives the coarse graining process. We use compression for all its vertical chains of arrows and finite-dimensional approximation for its horizontal ones. The coarse graining process decomposes global pictures, objects and properties into sequences of local ones together with a uniformity condition ensuring convergence of limits. For details on the notions of compression and finite-dimensional approximation, we refer to Subsection \ref{SSEC.NCDS_NCG_Notion}.\par


\pagebreak


Let $(\phi,\bpsi,\gamma,\nabla)$ be noncommutative differential structure for tracial AF-$C^{*}$-algebras $(A,\tau)$ and $(B,\omega)$ in $\lc{}f,\theta\rc$-setting. The coarse graining process applies to accessibility components. These may differ yet have states with identical fixed part. The latter are unique only in that each accessibility component has exactly one. For all fixed states $\xi\in\SII(A)$, note $3)$ in Proposition \ref{PRP.Wstar_Derivation_QG_HSG_Fixed_Part_I} implies $\CII_{A}(\xi)\subset\Fix_{A}(\xi)$ and decomposition

\begin{align}\label{EQ.SSEC.QOT_QIT_Encoding_13}
\textrm{Fix}_{A}(\xi)=\coprod_{\CII\subset\Fix_{A}(\xi)}\CII.
\end{align}

\begin{dfn}\label{DFN.QOT_Distance_AC_Fixed_Part}
Let $\xi\in\SII(A)$ be a fixed state. We say that $\CII\subset (\SII(A),\mathcal{W}_{\nabla}^{f,\theta})$ has fixed part $\xi$ if $\CII\subset\Fix_{A}(\xi)$.
\end{dfn}

\begin{rem}\label{REM.QOT_Distance_AC_Fixed_Part}
If $\CII\subset (\SII(A),\mathcal{W}_{\nabla}^{f,\theta})$, then the above shows $\CII$ has a unique fixed part $\xi$ as per Definition \ref{DFN.QOT_Distance_AC_Fixed_Part}. Yet $\xi$ is only unique among all $\mu\in\CII$. As such, we cannot exclude $\CII\neq\mathcal{C}_{A}(\xi)$ unless we intersect with a suitable convex subset of states, e.g.~$\mathcal{S}^{\NI,2}(A)$ as per $1)$ in Theorem \ref{THM.QOT_Distance_AC_L2}. This is classification and reason for $K$ in Diagram \ref{EQ.SSEC.QOT_QIT_Encoding_16}.
\end{rem}

The lowest horizontal chain of arrows in Diagram \ref{EQ.SSEC.QOT_QIT_Encoding_16} gives the coarse graining process for the following data. Let $\xi\in\SII(A)$ be a fixed state. For all $j\in\mathbb{N}$, we know $\bar{\xi}_{j}\in\SII(A_{j})$ is a fixed state if and only if $\xi_{j}\neq 0$. If $\xi_{j}\neq 0$ for $j\in\mathbb{N}$, then $\xi_{k}\neq 0$ for all $j\leq k$ in $\mathbb{N}$.
Let $j_{\min}\in\mathbb{N}$ minimal among all $j\in\mathbb{N}$ s.t.~$\xi_{j}\neq 0$. For all $j\geq j_{\min}$ in $\mathbb{N}$ and up to rescaling as per $1)$ in Definition \ref{DFN.AF_Cstar_Trace_Dualisation_Paths}, note $1.3)$ and $3)$ in Proposition \ref{PRP.Wstar_Derivation_QG_HSG_Fixed_Part_I} imply

\begin{align}\label{EQ.SSEC.QOT_QIT_Encoding_14}
\resj\lc\mathcal{F}_{A}(\xi)\rc{}=\mathcal{F}_{A_{j}}\lc\bar{\xi}_{j}\rc{}.
\end{align}

\noindent We rescale subsets of $\mathcal{F}_{A}(\xi)$ as per Equation \ref{EQ.SSEC.QOT_QIT_Encoding_14}. Let $K\subset\SII(A)$ be a convex subset s.t.~for all $j\geq j_{\min}$ in $\mathbb{N}$ and up to rescaling as per $1)$ in Definition \ref{DFN.AF_Cstar_Trace_Dualisation_Paths}, we have

\begin{align}\label{EQ.SSEC.QOT_QIT_Encoding_15}
\resj\lc\mathcal{C}\cap K\rc{}=\mathcal{C}_{A_{j}}\lc\bar{\xi}_{j}\rc{}.
\end{align}

\noindent Corollary \ref{COR.QOT_Distance_AC_L2}, which uses Theorem \ref{THM.QOT_Distance_AC_L2}, shows Equation \ref{EQ.SSEC.QOT_QIT_Encoding_15} is satisfied if $K$ equals $\SII(A)$, $\mathcal{S}^{\NI}(A)$, or $\mathcal{S}^{\NI,2}(A)$. Corollary \ref{COR.Rel_Ent_AF_Cstar_Trace} shows Equation \ref{EQ.SSEC.QOT_QIT_Encoding_15} is satisfied if $K$ is the domain of quantum relative entropy as per Definition \ref{DFN.Rel_Ent_AF_Cstar_Trace}. This lets us apply the coarse graining process in Chapter \ref{CH.L2W}. Theorem \ref{THM.QOT_Distance_AC_L2} and Theorem \ref{THM.QOT_Distance_AC_Rel_Ent} yield classification if $K$ equals $\mathcal{S}^{\NI,2}(A)$, resp.~the domain of quantum relative entropy.\par However, each choice implies restriction of the coarse graining process to suitable fixed states. If $K$ is the domain of quantum relative entropy, then our discussion in Section \ref{SEC.L2W_Rel_Ent} yields natural interpretation. For all $\mu\in\SII(A)$, $\Ent(\mu,\tau)\in [-\infty,\infty]$ is the relative entropy of $\mu$ w.r.t.~$\tau$ as per Equation \ref{EQ.DFN.Rel_Ent_AF_Cstar_Trace_1}. Theorem \ref{THM.Rel_Ent_AF_Cstar_Trace} ensures it measures information required to discriminate $\mu$ and $\tau$ through observation by extending its use from the strongly unital finite-trace case \lc{}cf.~pp.1-11 in \cite{BK.Ohy_Pet.1993.Rel_Ent}\rc{}. Restriction implies we only consider normal states, fixed or not, encoding a finite amount of information.\par
We assume data $\xi\in\SII(A)$ and $K$ as above. Using canonical inclusion maps for all vertical arrows, restriction maps for all uppermost horizontal arrows, as well as scaled restriction maps as per Equation \ref{EQ.SSEC.QOT_QIT_Encoding_14}, resp.~Equation \ref{EQ.SSEC.QOT_QIT_Encoding_15} for all lower horizontal arrows, we have diagram

\begin{equation}\label{EQ.SSEC.QOT_QIT_Encoding_16}
\begin{tikzcd}
A^{*}\arrow[rr, two heads] & & \ \cdots \arrow[r, two heads] & A_{j}^{*}\arrow[r, two heads] & \cdots \arrow[r, two heads] & A_{j_{\min}}^{*} \\
& & & & & \\
& & & & & \\
\mathcal{F}_{A}(\xi)\arrow[rr, two heads]\arrow[uuu,hook] & & \ \cdots \arrow[r, two heads] & \mathcal{F}_{A_{j}}\lc\bar{\xi}_{j}\rc\arrow[r, two heads]\arrow[uuu,hook] & \ \cdots \arrow[r, two heads] & \mathcal{F}_{A_{j_{\min}}}\lc\bar{\xi}_{j_{\min}}\rc\arrow[uuu,hook] \\
& & & & & \\
\mathcal{C}\cap K\arrow[rr, two heads]\arrow[uu,hook] & & \ \cdots \arrow[r, two heads] & \mathcal{C}_{A_{j}}\lc\bar{\xi}_{j}\rc\arrow[r, two heads]\arrow[uu,hook] & \cdots \arrow[r, two heads] & \mathcal{C}_{A_{j_{\min}}}\lc\bar{\xi}_{j_{\min}}\rc\arrow[uu,hook]
\end{tikzcd}
\end{equation}

\medskip

\noindent Diagram \ref{EQ.SSEC.QOT_QIT_Encoding_16} extends Diagram \ref{EQ.SSEC.QOT_QIT_Encoding_6}. Assuming a fixed state is necessary for having scaled restriction maps in Diagram \ref{EQ.SSEC.QOT_QIT_Encoding_16}. We use compression for each vertical chain of arrows in Diagram \ref{EQ.SSEC.QOT_QIT_Encoding_16} and finite-dimensional approximation for each horizontal one. This demands data compatible with both. Diagram \ref{EQ.SSEC.QOT_QIT_Encoding_16} relates a global picture given by the leftmost vertical chain of arrows to a sequence of local pictures given by vertical chains of arrows obtained as images of scaled restriction maps.\par
We explain our notion of \textit{locality}. For all $j\geq j_{\min}$ in $\mathbb{N}$, note $A_{j}^{*}\subset A^{*}$ restricts as per Equation \ref{EQ.SSEC.QOT_QIT_Encoding_12} to an equivalent formulation represented on a finite-dimensional model algebra of observables. We thereby restrict $\SII(A)$ to a standard representation of $\SII(A_{j})$ by conditioned testing on direct sums of full matrix algebras. We view the latter as local pictures in direct analogy to notions of locality for pure state spaces in the commutative setting, i.e.~locally compact Hausdorff spaces. Altogether, Diagram \ref{EQ.SSEC.QOT_QIT_Encoding_16} decomposes global pictures, objects and properties into sequences of local ones together with a uniformity condition ensuring convergence of limits.


\subsubsection*{Transport of information encoded in states on tracial AF-$\mathbf{C}^{*}$-algebras}

Let $(\phi,\bpsi,\gamma,\nabla)$ be noncommutative differential structure for tracial AF-$C^{*}$-algebras $(A,\tau)$ and $(B,\omega)$ in $\lc{}f,\theta\rc$-setting. Following our discussion in Subsection \ref{SSEC.QOT_QIT_Encoding}, we consider each $\mu\in\SII(A)$ as scaling limit of uniformly conditioned spin states encoding sequences of qubits. Note minimising geodesics do not restrict to other minimising geodesics in general. However, we expect a form of finite-dimensional approximation related to and well-behaved w.r.t.~the coarse graining process, at least for marginals, if transport of quantum information arises from quantum optimal transport. We therefore consider minimising geodesics approximated in finite dimensions in order to respect scaling limit description of marginals and fixed parts as above while retaining geodicity.\par
For all $\mu^{0},\mu^{1}\in\SII(A)$, Theorem \ref{THM.QOT_Minimiser_Approximation} shows we have $\mathcal{W}_{\nabla}^{\log}\lc\mu^{0},\mu^{1}\rc{}<\infty$ if and only if there exists $(\mu,w)\in\Geo\lc\mu^{0},\mu^{1}\rc$ approximated in finite dimensions by a sequence $\lc\mu^{j},w^{j}\rc_{j\geq m}\subset\Geo_{0}$. We moreover have 

\begin{align}\label{EQ.SSEC.QOT_QIT_Encoding_17}
\lc\mu^{j},w^{j}\rc\in\textrm{Geo}_{j}\big(\bar{\mu}_{j}^{0},\bar{\mu}_{j}^{1}\big)
\end{align}

\noindent for all $j\geq m$ and may pass to a subsequence converging to $(\mu,w)$ in $\Admnullone$ in this case. We consider each $\mu^{j}:[0,1]\longrightarrow\SII(A_{j})$ as optimal transport of uniformly conditioned spin states encoding qubits and therefore transport of quantum information. Corollary \ref{COR.QOT_Distance_AC_I} shows convergence to $(\mu,w)$ in $\Admnullone$ yields the global picture, here itself scaling limit w.r.t.~the coarse graining process, using a sequence of local pictures for transport of quantum information. Equation \ref{EQ.SSEC.QOT_QIT_Encoding_17} shows marginals are elements in the scaling limit sequence of marginals as per Equation \ref{EQ.SSEC.QOT_QIT_Encoding_7}.\par
We consider each $(\mu,w)\in\Geo$ approximated in finite-dimensions as optimal transport of scaling limits of uniformly conditioned spin states encoding sequences of qubits. We therefore view quantum optimal transport as transport of quantum information and say that it is compatible with the coarse graining process. Thus non-ergodicity restricts information-bearing degrees of freedom by the continuity equation, as visible from $3)$ in Proposition \ref{PRP.Wstar_Derivation_QG_HSG_Fixed_Part_I} in general, resp.~$2)$ in Proposition \ref{PRP.RM_Embedded_Submanifold_I} and Proposition \ref{PRP.RM_Embedded_Submanifold_II} upon coarse graining. Moreover, our description of transport of quantum information extends suitably to Example \ref{BSP.QOT_Second_Quantisation_Parametrised} and its generalisations.\par
We restrict to Example \ref{BSP.QOT_Second_Quantisation_Parametrised} and use its notation. The given state space $\SII\lc{}A_{X}\rc$ consists of normalised Radon measures on $X$ evaluating in $\AII(H)$ up to $C^{*}$-isometry as per Equation \ref{EQ.SSEC.QOT_DT_BSP_10}. Dualising the minimal $C^{*}$-tensor product \cite{BK.Kad_Rin.1997.OpAlg_II}\cite{BK.Tak.1979.OpAlg_I} yields

\begin{align}\label{EQ.SSEC.QOT_QIT_Encoding_18}
\SII\lc{}A_{X}\rc\cong\lset\mu\in C_{0}(X)^{*}\otimes\AII(H)^{*}\ \vset\ \mu\geq 0,\ \|\mu\|_{C_{0}(X)^{*}\otimes\AII(H)^{*}}=1\rset{},
\end{align}

\noindent where $C_{0}(X)^{*}\cong C_{c}(X)^{*}$ is the Banach space of totally finite signed Radon measures on $X$ by $\sigma$-compactness \lc{}cf.~Proposition 6.3.6 in \cite{BK.Ped.1989.Analysis_Now}\rc{}. Each gauge field $T\in X$ determines an encoding scheme of $\AII(H)_{+}^{*}$ as per Diagram \ref{EQ.SSEC.QOT_QIT_Encoding_16}. These vary since Example \ref{BSP.QOT_Second_Quantisation} applied to obtain each fibre depends entirely on the given inner fluctuation $D_{T}$ of $D$ as per Equation \ref{EQ.SSEC.QOT_DT_BSP_13}. If $\mu:[a,b]\longrightarrow\SII\lc{}A_{X}\rc$ is given by an admissible path s.t.~

\begin{align}\label{EQ.SSEC.QOT_QIT_Encoding_19}
\mu(t)=\delta_{\rho(t)}\otimes\nu(t)\in\lc{}C_{0}(X)^{*}\otimes\AII(H)^{*}\rc_{+}
\end{align}

\noindent under the isomorphism in Equation \ref{EQ.SSEC.QOT_QIT_Encoding_18} for a.e.~$t\in [a,b]$, then $\nu(t)\in\SII\lc\AII(H)\rc$ for a.e.~$t\in [a,b]$ as well. This suppresses encoding schemes. Upon considering said path in $\SII\lc{}A_{X}\rc$, i.e.~we know $t\mapsto\nu(t)\in\SII\lc\AII\lc{}H\lb{}J_{\rho(t)}\rb\rc\rc$ are states on varying CAR-algebras $t\mapsto \AII\lc{}H\lb{}J_{\rho(t)}\rb\rc$ as per Equation \ref{EQ.BSP.QOT_Second_Quantisation_Parametrised_1}, we see minimising geodesics transporting Dirac measures are transport of quantum information under varying encoding schemes. We are therefore motivated, in direct analogy to the classical case \cite{BK.Amb_Gig_Sav.2008.Classical_OT_GradFlow}\cite{ART.Dol_Naz_Sav.2009.Generalised_OT}\cite{BK.Vil.2009.OT} generalising from transport of point mass to transport of mass distributions, to view parametrised quantum optimal transport as transport of densities of quantum information over those encoding schemes of $\AII(H)_{+}^{*}$ as per Diagram \ref{EQ.SSEC.QOT_QIT_Encoding_16} parametrised by $X$.


\chapter[Metric Geometry of Quantum $L^{2}$-Wasserstein Distances]{Metric Geometry of Quantum $\mathbf{L}^{2}$-Wasserstein Distances}\label{CH.L2W}

The logarithmic mean setting uses the logarithmic operator mean for interpolation\linebreak parameter one. This defines quantum $L^{2}$-Wasserstein distances in direct analogy to the classical case \cite{ART.Dol_Naz_Sav.2009.Generalised_OT}. The logarithmic operator mean is characterised as the one inducing the Kubo-Mori-Bogoliubov inner product \cite{ART.Pet_Tot.1993.Inner_Product}. Up to coarse graining in the logarithmic mean setting, the given noncommutative chain rule ensures heat flow is gradient flow of quantum relative entropy. In our logarithmic mean setting, which does assume the AF-$C^{*}$-setting, yet neither ergodicity nor finite trace, we extend results in \cite{ART.Car_Maa.2014.Quantum_OT_I}\cite{ART.Car_Maa.2017.Quantum_OT_II}\cite{ART.Car_Maa.2020.Quantum_OT_III} and \cite{ART.Erb_Maa.2012.Discrete_OT_Ricci_Bounds} to the general case and view lower Ricci bounds as measurement convexity of quantum information. Non-ergodicity and non-finite trace ensure fundamental example classes in Subsection \ref{SSEC.QOT_DT_BSP} are covered. We summarise our contributions below.\par
We extend quantum relative entropy in the sense of Araki \cite{ART.Ara.1975.Rel_Ent_I}\cite{ART.Ara.1977.Rel_Ent_II} and Umegaki \cite{ART.Ume.1962.Rel_Ent} to the AF-$C^{*}$-setting. Note our construction ensures it measures information required to discriminate a given state and, possibly non-finite, trace through observation by extending its use from the strongly unital finite-trace case \cite{BK.Ohy_Pet.1993.Rel_Ent}. If $\EVI_{\lambda}$-gradient flow of quantum relative entropy exist, then it is heat flow. We show claimed equivalence of $\EVI_{\lambda}$-gradient flow, $\lambda$-convexity, Bakry-\'Emery and Hessian lower bound conditions by means of the coarse graining process. We then define lower Ricci bounds of quantum\linebreak gradients using any one of said equivalent conditions, give sufficient conditions for lower Ricci bounds of direct sum quantum gradients and, assuming lower Ricci bounds, derive functional inequalities $\HWI_{\lambda}$, $\MLSI_{\lambda}$ and $\TW_{\lambda}$ in the AF-$C^{*}$-setting.\par
We view quantum Laplacians as generators of quantum noise evolution in order to have non-spatiality of lower Ricci bounds and associated energy-information trade-offs. Following Landauer's principle \cite{ART.Lan.1961.Information_Physical_I}\cite{ART.Lan.1961.Information_Physical_II} and its extension to quantum information theory \cite{BK.Cam_Def.2019.Quantum_StM_Information}\cite{ART.DiVi_Loss.1998.Quantum_Information_Physical}, erasure of quantum information implies strictly positive production of quantum entropy. Yet it is unclear how the $\EVI_{\lambda}$-gradient flow property selects noise diffusion terms, i.e.~generators of quantum noise evolution, in our case. To this end, we formulate a maximum entropy production principle \cite{ART.Dew.2003.MaxEnt_Information_I}\cite{ART.Dew.2005.MaxEnt_Information_II}\cite{ART.Mar_Sel.2006.MaxEnt_Review}. We show quantum Laplacians satisfy, up to sign, a quantum Fokker-Planck equation with vanishing drift term in scaling limit, i.e.~only noise diffusion term. Altogether, we obtain a description of quantum Laplacians in terms of both quantum statistical mechanics \cite{BK.Bra.1987.OpAlg_Quantum_StM_I}\cite{BK.Bra.1987.OpAlg_Quantum_StM_II} and quantum information theory \cite{BK.Nie_Chu.2000.Quantum_Computation_Information} as claimed in the introduction of Chapter \ref{CH.QOT}.\par


\newpage


\noindent\textbf{Structure.} In Section \ref{SEC.L2W_Rel_Ent}, we discuss quantum relative entropy. We extend to, possibly non-finite, traces in the second variable. In Section \ref{SEC.L2W_Log_Mean}, we discuss the logarithmic mean setting and quantum $L^{2}$-Wasserstein distances. Moreover, we formulate our maximum entropy production principle. In Section \ref{SEC.L2W_EVI}, we consider heat flow as $\EVI_{\lambda}$-gradient flow of quantum relative entropy and show our equivalence theorem. We discuss non-spatial lower Ricci bounds and energy-information trade-offs parametrised by lower bounds on quantum noise, give sufficient conditions and derive functional inequalities.


\section{Quantum relative entropy}\label{SEC.L2W_Rel_Ent}

Quantum relative entropy is an extension of relative entropy for tracial $C^{*}$-algebras to the AF-$C^{*}$-setting. We construct it by extending Kosaki's formula \cite{BK.Ohy_Pet.1993.Rel_Ent} to traces in the second variable. Relative entropy for tracial $C^{*}$-algebras is the fundamental example of quasi-entropies and therefore quantum $f$-divergences \cite{ART.Hia.2018.QFD_I}\cite{ART.Hia.2019.QFD_II}. We also know it measures information required to discriminate two given states through observation \cite{BK.Ohy_Pet.1993.Rel_Ent}. Since it is given by extension of Kosaki's formula, our construction ensures quantum relative entropy likewise measures information required to discriminate a given state and, possibly non-finite, trace through observation.\par
In Subsection \ref{SSEC.L2W_EVI_Equivalence}, we consider heat flow as $\EVI_{\lambda}$-gradient flow of quantum relative entropy. This uses two most essential properties of quantum relative entropy. First, we show the latter is compatible with compression and finite-dimensional approximation. Secondly, we show it satisfies a suitable notion of l.s.c.~in topology of the given quantum optimal transport distance. However, finite-dimensional approximation and l.s.c.~do not hold for all states in general. The latter requires strong unitality and finite trace. Upon restriction to finitely supported accessibility components, i.e.~having finitely supported fixed state, we satisfyingly recover the strongly unital finite-trace case depending on the given finitely supported fixed state by compressing with uniform majorants of their local support. Examples of finitely supported fixed states arise from fixed states on tracial AF-$C^{*}$-algebras generating hyperfinite factors of type I and II by $\sigma$-weak closure.

\medskip

\noindent\textbf{Structure.} In Subsection \ref{SSEC.L2W_Rel_Ent_AF}, we review relative entropy for $C^{*}$-algebras expressed using Kosaki's formula. We construct quantum relative entropy by extending to traces in the second variable. In Subsection \ref{SSEC.L2W_Rel_Ent_AC}, we discuss uniform majorisation, finitely supported fixed states and show all properties required of quantum relative entropy.


\subsection[Quantum relative entropy for tracial AF-$C^{*}$-algebras]{Quantum relative entropy for tracial AF-$\mathbf{C}^{*}$-algebras}\label{SSEC.L2W_Rel_Ent_AF}

Theorem 5.11 in \cite{BK.Ohy_Pet.1993.Rel_Ent} states Kosaki's formula. It is a variational expression of relative entropy for normal positive bounded functionals on $W^{*}$-algebras w.r.t.~each other. This determines relative entropy for $W^{*}$-algebras. We construct quantum relative entropy by two consecutive extensions of Kosaki's formula. First, we extend to positive bounded functionals on $C^{*}$-algebras by evaluating their canonical normal extensions to universal enveloping $W^{*}$-algebras. This determines relative entropy for $C^{*}$-algebras. Secondly, we extend to positive bounded functionals on tracial AF-$C^{*}$-algebras w.r.t.~the given trace.\par
This determines quantum relative entropy. Lemma \ref{LEM.Rel_Ent_AF_Cstar_Trace_II} recovers Kosaki's formula as per Theorem 5.11 in \cite{BK.Ohy_Pet.1993.Rel_Ent} for normal positive bounded functionals with integrable support. Standard reference for Kosaki's formula, as well as relative entropy for $C^{*}$-~and $W^{*}$-algebras alike, is \cite{BK.Ohy_Pet.1993.Rel_Ent}. We refer to pp.35-36 and pp.98-99 in \cite{BK.Ohy_Pet.1993.Rel_Ent} for a review.


\subsubsection*{Relative entropy for tracial $C^{*}$-algebras}

Umegaki defined relative entropy for semi-finite $W^{*}$-algebras \cite{ART.Ume.1962.Rel_Ent}. Using relative modular operators, Araki generalised to all $W^{*}$-algebras \cite{ART.Ara.1975.Rel_Ent_I}\cite{ART.Ara.1977.Rel_Ent_II}. Equation \ref{EQ.DFN.Rel_Ent_Wstar_1} is Kosaki's formula as per Theorem 5.11 in \cite{BK.Ohy_Pet.1993.Rel_Ent}. Using universal enveloping $W^{*}$-algebras in Kosaki's formula, we engage in our first extension by adapting constructions in \cite{BK.Ohy_Pet.1993.Rel_Ent} but with additional detail required for our second one. Assuming tracial $C^{*}$-algebra, Lemma \ref{LEM.Rel_Ent_Cstar_Normality} shows Kosaki's formula uses spaces of bounded measurable operators. Proposition \ref{PRP.Rel_Ent_Cstar_I} and Proposition \ref{PRP.Rel_Ent_Cstar_II} collect properties. We consider two instructive examples here. Example \ref{BSP.Rel_Ent_Cstar_Fin_I} gives the finite-dimensional setting. Example \ref{BSP.Rel_Ent_Cstar_Fin_II} shows necessity of strong unitality.\par
Let $(M,\tau)$ be a tracial $W^{*}$-algebra and $A\subset M$ a $\sigma$-weakly dense $C^{*}$-subalgebra. Ergo $M=L^{\infty}(A,\tau)$ and $M_{*}=L^{1}(A,\tau)$. Following Remark \ref{REM.Wstar_Trace_MSP}, we have $L^{1}(A,\tau)^{\flat}\subset A^{*}$ as partially ordered Banach spaces.

\begin{dfn}\label{DFN.Wstar_Step_Function}
Let $V\subset L^{\infty}(A,\tau)$ be a linear subspace s.t.~$1_{A}\in V$. Let $n\in\mathbb{N}$.

\begin{itemize}
\item[1)] Let $\mathcal{T}_{n}(V)$ be the set of all step functions $F:(n^{-1},\infty)\longrightarrow V$ s.t.~$\absv{1.15}{\im F}<\infty$. Using the constant map $t\mapsto 1_{M}=1_{A}$ on $(n^{-1},\infty)$, set $F^{\perp}:=1_{A}-F$ for all $F\in\mathcal{T}_{n}(V)$.

\item[2)] $\mathcal{T}_{n}^{u}(V):=\big\{\hspace{0.025cm} F\in\mathcal{T}_{n}(V)\ \vset\ \exists t\in (n^{-1},\infty)\ \forall s\geq t:\ F(s)=1_{A}\hspace{0.025cm} \big\}$.
\end{itemize}
\end{dfn}

Definition \ref{DFN.Rel_Ent_Wstar} gives the relative entropy $\Ent:L^{1}(A,\tau)_{+}^{\flat}\times L^{1}(A,\tau)_{+}^{\flat}\longrightarrow (-\infty,\infty]$. Equation \ref{EQ.DFN.Rel_Ent_Wstar_1} is Kosaki's formula which we extend to variational expressions using positive bounded functionals on $C^{*}$-algebras, and w.r.t.~traces in the second variable. We call extensions relative entropy, resp.~Kosaki's formula as well. All extensions coincide on intersections of domains. For all $\mu,\eta\in L^{1}(A,\tau)_{+}^{\flat}$, note $\Ent(\mu,\eta)$ measures information required to discriminate $\mu$ and $\eta$ through observation \lc{}cf.~pp.1-11 in \cite{BK.Ohy_Pet.1993.Rel_Ent}\rc{}. As expected in the commutative setting, Umegaki's definition shows Kosaki's formula yields relative entropy of probability densities, i.e.~Kullback-Leibler divergence \lc{}cf.~pp.35-36 in \cite{BK.Ohy_Pet.1993.Rel_Ent}\rc{}. Theorem \ref{THM.Rel_Ent_AF_Cstar_Trace} extends the above notion of discriminating information.

\begin{dfn}\label{DFN.Rel_Ent_Wstar}
For all $\mu\in L^{1}(A,\tau)_{+}^{\flat}$, set $\|x\|_{\mu}:=\sqrt{\mu(x^{*}x)}$ for all $x\in L^{\infty}(A,\tau)$. For all $\mu,\eta\in L^{1}(A,\tau)_{+}^{\flat}$, the relative entropy of $\mu$ w.r.t.~$\eta$ is defined by

\begin{align}\label{EQ.DFN.Rel_Ent_Wstar_1}
\Ent(\mu,\eta):=\sup_{\substack{n\in\mathbb{N},\\ F\in\mathcal{T}_{n}(L^{\infty}(A,\tau))}}\left\{\|\mu\|_{A^{*}}\log n-\int_{n^{-1}}^{\infty}t^{-1}\dblv{}F^{\perp}(t)\dblv_{\mu}^{2}+t^{-2}\dblv{}F(t)^{*}\dblv_{\eta}^{2}~dt\right\}.
\end{align}
\end{dfn}

\begin{rem}\label{REM.Rel_Ent_Wstar}
Let $V\subset L^{\infty}(A,\tau)$ be a strong$^{*}$-dense linear subspace s.t.~$1_{A}\in V$. Then Theorem 5.11 in \cite{BK.Ohy_Pet.1993.Rel_Ent} shows we may replace the supremum over all $\mathcal{T}_{n}(L^{\infty}(A,\tau))$ with the one over all $\mathcal{T}_{n}(V)$ in Kosaki's formula, hence the one over all $\mathcal{T}_{n}^{u}(V)$. We use this throughout our discussion.
\end{rem}

We review properties of relative entropy for $W^{*}$-algebras. We take the supremum over all $\mathcal{T}_{n}^{u}\lc{}L^{\infty}(A,\tau)\rc$ in Kosaki's formula and apply Fatou's lemma. Kosaki's formula therefore shows the relative entropy is jointly convex and l.s.c.~in $w^{*}$-topology given by $L^{\infty}(A,\tau)=L^{1}(A,\tau)^{*}$. Let $\mu,\eta\in L^{1}(A,\tau)_{+}^{\flat}$. If $\mu,\eta\neq 0$, then Proposition 5.1 in \cite{BK.Ohy_Pet.1993.Rel_Ent} shows

\begin{align}\label{EQ.SSEC.L2W_Rel_Ent_AF_1}
\Ent(\mu,\eta)\geq\lc\log \|\mu\|_{A^{*}}-\log \|\eta\|_{A^{*}}\rc{}\cdot \|\mu\|_{A^{*}}>-\infty
\end{align}

\noindent as $\|\mu\|_{A^{*}},\|\eta\|_{A^{*}}\in (0,\infty)$. Kosaki's formula further implies $\Ent(0,\eta)=0$ and $\Ent(\mu,0)=\infty$ in general \lc{}cf.~proof of Proposition \ref{PRP.Rel_Ent_Cstar_I}\rc{}. If $N\subset L^{\infty}(A,\tau)$ is a unital $W^{*}$-subalgebra, then Corollary 5.12 in \cite{BK.Ohy_Pet.1993.Rel_Ent} shows we have restriction

\begin{align}\label{EQ.SSEC.L2W_Rel_Ent_AF_2}
\Ent(\mu,\eta)\geq\Ent\lc\mu\vert_{N},\eta\vert_{N}\rc{}
\end{align}

\noindent since unital $W^{*}$-algebra inclusions are normal unital Schwarz maps. Altogether, we know $\Ent:L^{1}(A,\tau)_{+}^{\flat}\times L^{1}(A,\tau)_{+}^{\flat}\longrightarrow (-\infty,\infty]$ is jointly convex, l.s.c.~in $w^{*}$-topology of $L^{\infty}(A,\tau)$ and has restriction property as per Equation \ref{EQ.SSEC.L2W_Rel_Ent_AF_2}. Moreover, we may replace suprema in Kosaki's formula as per Remark \ref{REM.Rel_Ent_Wstar}.\par
If $(A,\tau)$ is a strongly unital AF-$C^{*}$-algebra with finite trace, then the relative entropy satisfies the following consequence of the martingale property \lc{}cf.~iv\rc{} in Corollary 5.12 in \cite{BK.Ohy_Pet.1993.Rel_Ent}\rc{}. For all $\mu\in L^{1}(A,\tau)_{+}^{\flat}$, we have finite-dimensional approximation

\begin{align}\label{EQ.SSEC.L2W_Rel_Ent_AF_3}
\Ent(\mu,\tau)=\lim_{j\in\mathbb{N}}\hspace{0.025cm} \Ent\lc\mu_{j},\tau_{j}\rc{}.
\end{align}

\noindent The martingale property requires l.s.c.~in $w^{*}$-topology of $L^{\infty}(A,\tau)$ and Equation \ref{EQ.SSEC.L2W_Rel_Ent_AF_2} for each generating $C^{*}$-subalgebra. If we extend to, possibly non-finite, traces in the second variable, then either may fail. Following Remark \ref{REM.Rel_Ent_AF_Cstar_Trace_I}, l.s.c.~in $w^{*}$-topology of $L^{\infty}(A,\tau)$ fails in general if the trace is non-finite and the relative entropy takes negative infinity as value. Example \ref{BSP.Rel_Ent_Cstar_Fin_II} shows Equation \ref{EQ.SSEC.L2W_Rel_Ent_AF_2} may fail if $(A,\tau)$ is not strongly unital. Uniform majorisation of local support suffices to prevent failure and recover finite-dimensional approximation property as per Equation \ref{EQ.SSEC.L2W_Rel_Ent_AF_3}. Theorem \ref{THM.QOT_Distance_AC_FS} shows the latter on finitely supported accessibility components.\par
Definition \ref{DFN.Rel_Ent_Cstar} extends Definition \ref{DFN.Rel_Ent_Wstar} to $\Ent:A_{+}^{*}\times A_{+}^{*}\longrightarrow (-\infty,\infty]$. We require the following. We have separable Hilbert space $H_{U}$, universal faithful $^{*}$-representation $\pi_{U}:A\longrightarrow\BII\lc{}H_{U}\rc$, and universal enveloping $W^{*}$-algebra $U(A):=\pi_{U}(A)''$ of $A$ \cite{BK.Tak.1979.OpAlg_I}. For all $\mu\in A_{+}^{*}$, get unique $U(\mu)\in U(A)_{*,+}\subset U(A)_{+}^{*}$ s.t.~$U(\mu)\vert_{A}=\mu$. These are called canonical normal extensions. Note $\dblv{}U(\mu)\dblv_{U(A)^{*}}=\|\mu\|_{A^{*}}$ in each case by construction.

\begin{dfn}\label{DFN.Rel_Ent_Cstar}
For all $\mu,\eta\in A_{+}^{*}$, the relative entropy of $\mu$ w.r.t.~$\eta$ is defined by

\begin{align}\label{EQ.DFN.Rel_Ent_Cstar_1}
\Ent(\mu,\eta):=\sup_{\substack{n\in\mathbb{N},\\ F\in\mathcal{T}_{n}(U(A))}}\left\{\|\mu\|_{A^{*}}\log n-\int_{n^{-1}}^{\infty}t^{-1}\dblv{}F^{\perp}(t)\dblv_{U(\mu)}^{2}+t^{-2}\dblv{}F(t)^{*}\dblv_{U(\eta)}^{2}~dt\right\}.
\end{align}
\end{dfn}

\begin{rem}\label{REM.Rel_Ent_Cstar}
Note Definition \ref{DFN.Rel_Ent_Wstar}, as well as those properties of relative entropy for $W^{*}$-algebras given above, do not require traciality. Definition \ref{DFN.Rel_Ent_Cstar} therefore gives relative entropy for $U(A)$. We use this throughout our discussion.
\end{rem}

\begin{prp}\label{PRP.Rel_Ent_Cstar_I}
For all $\mu,\eta\in A_{+}^{*}$ and $a,b>0$ in $\mathbb{R}$, we have

\begin{itemize}
\item[1)] $\Ent(a\mu,b\eta)=a\Ent(\mu,\eta)+a\big(\log a-\log b\big)\cdot \|\mu\|_{A^{*}}$,

\item[2)] $\Ent(\mu,\eta)\geq\lc\log \|\mu\|_{A^{*}}-\log \|\eta\|_{A^{*}}\rc\cdot \|\mu\|_{A^{*}}$ if $\eta\neq 0$,

\item[3)] $\Ent(0,\eta)=0$ and $\Ent(\mu,0)=\infty$ if $\mu\neq 0$,

\item[4)] $\Ent(\mu,\eta)>-\infty$.
\end{itemize}
\end{prp}
\begin{proof}
Let $\mu,\eta\in A_{+}^{*}$ and $a,b>0$ in $\mathbb{R}$. Proposition 5.1 in \cite{BK.Ohy_Pet.1993.Rel_Ent} is $1)$ and $2)$. Kosaki's formula implies $\Ent(0,\eta)=0$ by selecting $F=0$ for all $n\in\mathbb{N}$ in order to estimate the supremum. If $\mu\neq 0$, then Kosaki's formula likewise implies $\Ent(\mu,0)=\infty$ by selecting $F=1_{U(A)}$ in each case. Get $3)$. We see $2)$ and $3)$ imply $4)$ at once.
\end{proof}

We have $\sigma$-weakly dense $C^{*}$-subalgebras $A\subset U(A)$ and $A\subset L^{\infty}(A,\tau)$. Universal property implies there exists unique normal $^{*}$-homomorphism $\varphi:U(A)\longrightarrow L^{\infty}(A,\tau)$ s.t.~$\varphi\circ\pi_{U}=\id_{A}$. It is unital and surjective, further mapping the unit ball in $U(A)$ to the one in $L^{\infty}(A,\tau)$ as per Remark \ref{REM.Rel_Ent_Cstar_Normality}. We define normal trace $U\lc\tau\rc$ on $U(A)$ by setting

\begin{align}\label{EQ.SSEC.L2W_Rel_Ent_AF_4}
U\lc\tau\rc{}(x):=\tau\lc\varphi(x)\rc{}    
\end{align}

\noindent for all $x\in U(A)_{+}$. We neither claim nor use semi-finiteness.

\begin{rem}\label{REM.Rel_Ent_Cstar_Normality}
Since $\varphi\vert_{A}=\id_{A}$, the Kaplansky density theorem shows $\varphi$ maps the unit ball in $U(A)$ to the one in $L^{\infty}(A,\tau)$ \lc{}cf.~Theorem 5.3.5 in \cite{BK.Kad_Rin.1997.OpAlg_I}\rc{}. Thus $\varphi$ is surjective. It is unital by normality and Proposition \ref{PRP.AF_Cstar_Unit}.
\end{rem}

Lemma \ref{LEM.Rel_Ent_Cstar_Normality} ensures Definition \ref{DFN.Rel_Ent_Cstar} is well-behaved w.r.t.~normality. We use the following. For all $^{*}$-subalgebras of $W^{*}$-algebras, closure in strong and weak topology are equivalent. Such closures are equivalent to closure w.r.t.~bounded strong, as well as bounded weak convergence \lc{}cf.~Proposition \ref{PRP.Wstar_BdCon}\rc{}. Note $\lc\sigma\textrm{-}\rc$weak-~and $w^{*}$-convergence coincide on bounded sets \lc{}cf.~Lemma II.2.5 in \cite{BK.Tak.1979.OpAlg_I} and Proposition \ref{PRP.Wstar_Equivalence}\rc{}. Bounded sets in tracial $W^{*}$-algebras are compact in $w^{*}$-topology, ergo weakly compact.

\begin{lem}\label{LEM.Rel_Ent_Cstar_Normality}
For all $\mu\in A_{+}^{*}$, the following are equivalent:

\begin{itemize}
\item[1)] There exists unique normal extension of $\mu$ to $L^{\infty}(A,\tau)$ s.t.~$U(\mu)=\mu\circ\varphi$.

\item[2)] For all projections $p\in U(A)$, $U(\mu)\lc{}p\rc{}=0$ if $U\lc\tau\rc\lc{}p\rc{}=0$.

\item[3)] $\ker\varphi\subset\ker U(\mu)$.

\item[4)] $\mu\in L^{1}(A,\tau)_{+}^{\flat}$.
\end{itemize}
\end{lem}
\begin{proof}
Note Remark \ref{REM.Rel_Ent_Cstar_Normality}. Let $\mu\in A_{+}^{*}$. For all projections $p\in U(A)$, faithfulness of $\tau$ implies $U\lc\tau\rc\lc{}p\rc{}=0$ if and only if $\varphi\lc{}p\rc{}=0$. Thus $3)$ implies $2)$. We know $U(\mu)$ and $\varphi$ are completely positive normal maps \lc{}cf.~Example \ref{BSP.Wstar_CP_I} and Example \ref{BSP.Wstar_CP_II}\rc{}, and therefore bounded weakly continuous by normality \lc{}cf.~Proposition \ref{PRP.Wstar_Normal}\rc{}. Ergo $\ker\varphi$ is a $W^{*}$-subalgebra. Note $W^{*}$-algebras are bounded weakly generated by their projections \lc{}cf.~Proposition \ref{PRP.Wstar_Generated}\rc{}. Hence $2)$ implies $3)$. Altogether, get equivalence of $2)$ and $3)$.\par
Clearly, $1)$ implies $2)$. Assume $\ker\varphi\subset\ker U(\mu)$. For all $x\in L^{\infty}(A,\tau)$, get $\varphi^{-1}(x)\neq\emptyset$ by surjectivity and set $\mu(x):=U(\mu)(y)$ for fixed but arbitrary $y\in\varphi^{-1}(x)$. This is independent of our choice as $3)$ ensures $\ker\varphi\subset\ker U(\mu)$. We thereby define a positivity-preserving linear map $\mu:L^{\infty}(A,\tau)\longrightarrow\mathbb{C}$ s.t.~$U(\mu)=\mu\circ\varphi$. Thus $\|\mu\|_{L^{\infty}(A,\tau)^{*}}=\dblv{}U(\mu)\dblv_{U(A)^{*}}=\|\mu\|_{A^{*}}$ since $\varphi$ is unital, hence we have extension $\mu\in L^{\infty}(A,\tau)_{+}^{*}$. If $x=\bdw$-$\lim_{k\in K}x_{k}$ implies $\lim_{k\in K}\absv{1}{\mu\lc{}x-x_{k}\rc{}}=0$ for all nets $\{x_{k}\}_{k\in K}\subset L^{\infty}(A,\tau)$, then complete positivity of $\mu$ shows its normality \lc{}cf.~Example \ref{BSP.Wstar_CP_I} and Proposition \ref{PRP.Wstar_Normal}\rc{}. Let $x=\bdw$-$\lim_{k\in K}x_{k}$. By considering all accumulation points of $\lset\mu(x_{k})\rset_{k\in K}\subset\mathbb{R}$ and showing they are in fact zero as claimed above, we assume $\lim_{k\in K}\absv{1}{\mu\lc{}x-x_{k}\rc{}}$ exists without loss of generality.\par
Since $\varphi$ is surjective on unit balls, we have both weakly convergent bounded subnet $\{x_{k}\}_{k\in K}\subset L^{\infty}(A,\tau)$ and weakly convergent bounded net $\{\hspace{0.0125cm} y_{k}\hspace{0.0025cm}\}_{k\in K}\subset U(A)$ s.t.~$x_{k}=\varphi\lc{}y_{k}\rc$ for all $k\in K$. Set $y:=\bdw$-$\lim_{k\in K}y_{k}$. Get $x=\varphi(y)$ by normality of $\varphi$. Thus $\lim_{k\in K}\mu\lc{}x-x_{k}\rc{}=\lim_{k\in K}U(\mu)\lc{}y-y_{k}\rc{}=0$ by normality of $U(\mu)$, hence $\mu\in L^{\infty}(A,\tau)_{+}^{*}$ is normal as discussed above and therefore a unique extension as required. Ergo $1)$ implies $2)$. Altogether, get equivalence of $1)$ and $2)$. Note Remark \ref{REM.Wstar_Trace_MSP}. In particular, $L^{1}(A,\tau)_{+}^{\flat}\subset A_{+}^{*}$ is determined by normality. Thus $1)$ implies $4)$. Assume $\mu\in L^{1}(A,\tau)_{+}^{\flat}$. We obtain $U(\mu)\vert_{A}=\mu\circ\varphi\vert_{A}$ by construction. Normality of $\mu$ and $\varphi$ extends the latter identity to $U(A)$. Hence $4)$ implies $1)$. Altogether, get equivalence of $1)$ and $4)$. All statements are equivalent.
\end{proof}

Proposition \ref{PRP.Rel_Ent_Cstar_II} collects further properties. Lemma \ref{LEM.Rel_Ent_Cstar_Normality} implies Equation \ref{EQ.PRP.Rel_Ent_Cstar_II_2} and therefore Equation \ref{EQ.PRP.Rel_Ent_Cstar_II_1}. Example \ref{BSP.Rel_Ent_Cstar_Fin_I} gives the finite-dimensional setting. Quantum entropy is negative quantum relative entropy. Example \ref{BSP.Rel_Ent_Cstar_Fin_II} shows Equation \ref{EQ.SSEC.L2W_Rel_Ent_AF_2} may fail in the finite-dimensional setting if $(A,\tau)$ is not strongly unital.

\begin{prp}\label{PRP.Rel_Ent_Cstar_II}
$\Ent:L^{1}(A,\tau)_{+}^{\flat}\times L^{1}(A,\tau)_{+}^{\flat}\longrightarrow (-\infty,\infty]$ is jointly convex and l.s.c.~in $w^{*}$-topology of $A[1_{A}]^{*}$. Furthermore, $\Ent$ satisfies the following.

\begin{itemize}
\item[1)] For all $\mu,\eta\in L^{1}(A,\tau)_{+}^{\flat}$, we have

\begin{align}\label{EQ.PRP.Rel_Ent_Cstar_II_1}
\Ent(\mu,\eta)=\sup_{\substack{n\in\mathbb{N},\\ F\in\mathcal{T}_{n}^{u}(A[1_{A}])}}\left\{\|\mu\|_{A^{*}}\log n-\int_{n^{-1}}^{\infty}t^{-1}\dblv{}F^{\perp}(t)\dblv_{\mu}^{2}+t^{-2}\dblv{}F(t)^{*}\dblv_{\eta}^{2}~dt\right\}.
\end{align}

\begin{reapply}
\end{reapply}

\item[2)] Let $N\subset L^{\infty}(A,\tau)$ be a unital $W^{*}$-subalgebra. For all $\mu,\eta\in L^{1}(A,\tau)_{+}^{\flat}$, we have

\begin{align}
\Ent(\mu,\eta)\geq\Ent\lc\mu\vert_{N},\eta\vert_{N}\rc{}.
\end{align}

\begin{reapply}
\end{reapply}

\end{itemize}
\end{prp}
\begin{proof}
Note $4)$ in Proposition \ref{PRP.Rel_Ent_Cstar_I} shows $\Ent>-\infty$ on norm bounded sets. Kosaki's formula implies $\Ent$ is jointly convex. Let $\mu,\eta\in L^{1}(A,\tau)_{+}^{\flat}$. We know $U(\mu)=\mu\circ\varphi$ and $U(\eta)=\eta\circ\varphi$ by Lemma \ref{LEM.Rel_Ent_Cstar_Normality}. Since $\varphi$ is a unital surjective $^{*}$-homomorphism, mapping $\mathcal{T}_{n}^{u}(U(A))$ to $\mathcal{T}_{n}^{u}\lc{}L^{\infty}(A,\tau)\rc$ via $F\mapsto G:=\varphi\circ F$ for all $n\in\mathbb{N}$ shows

\begin{align}\label{EQ.PRP.Rel_Ent_Cstar_II_2}
\Ent(\mu,\eta)=\sup_{\substack{n\in\mathbb{N},\\ F\in\mathcal{T}_{n}^{u}\lc{}L^{\infty}(A,\tau)\rc{}}}\left\{\|\mu\|_{A^{*}}\log n-\int_{n^{-1}}^{\infty}t^{-1}\dblv{}G^{\perp}(t)\dblv_{\mu}^{2}+t^{-2}\dblv{}G(t)^{*}\dblv_{\eta}^{2}~dt\right\}.
\end{align}

\noindent Since $\mu,\eta\in L^{\infty}(A,\tau)^{*}$, Equation \ref{EQ.PRP.Rel_Ent_Cstar_II_2} shows $\Ent(\mu,\eta)$ is the relative entropy of $\mu$ w.r.t.~$\eta$ as per Definition \ref{DFN.Rel_Ent_Wstar}. Get $1)$ by replacing $L^{\infty}(A,\tau)$ with $A[1_{A}]$ in the second suprema of the equation. Applying Fatou's lemma to Equation \ref{EQ.PRP.Rel_Ent_Cstar_II_1} then shows l.s.c.~in $w^{*}$-topology of $A[1_{A}]^{*}$. Equation \ref{EQ.SSEC.L2W_Rel_Ent_AF_2} and Equation \ref{EQ.PRP.Rel_Ent_Cstar_II_2} show $2)$ immediately.
\end{proof}

\begin{bsp}\label{BSP.Rel_Ent_Cstar_Fin_I}
Assume $(A,\tau)$ is finite-dimensional. Following Proposition \ref{PRP.AF_Cstar_Trace_II}, we moreover assume $(A,\tau)=\lc{}M_{n}(\mathbb{C}),\tr_{n}\rc$ for $n\in\mathbb{N}$ without loss of generality. The general finite-dimensional case is therefore given by a weighted sum of terms having following form up to pull-back along $C^{*}$-isometries as per Equation \ref{EQ.NTN.AF_Cstar_Fin_Isometry_1}. For all $\mu,\eta\in M_{n}(\mathbb{C})_{+}^{*}$, the relative entropy of $\mu$ w.r.t.~$\eta$ is given by

\begin{align*}
\Ent(\mu,\eta)=
\begin{cases}
0 & \If\ \mu=0, \\
\textrm{tr}_{n}\lc\sharp\mu\cdot \lc\log\sharp\mu-\log\sharp\eta\rc\rc & \If\ \mu\neq 0\ \textrm{and}\ \supp\mu\leq\supp\eta, \\
\infty & \Else.
\end{cases}
\end{align*}

\noindent The above characterisation is Umegaki's definition, except we make vanishing for $\mu=0$ explicit. It generalises to Araki's definition \lc{}cf.~p.77 in \cite{BK.Ohy_Pet.1993.Rel_Ent}\rc{}, which in turn coincides with Kosaki's formula by Theorem 5.11 in \cite{BK.Ohy_Pet.1993.Rel_Ent}. The negative of Umegaki's definition is quantum entropy, i.e.~von Neumann entropy \lc{}cf.~p.17 in \cite{BK.Ohy_Pet.1993.Rel_Ent}\rc{}. Corollary \ref{COR.Rel_Ent_AF_Cstar_Trace} extends such description to the general case.
\end{bsp}

\begin{bsp}\label{BSP.Rel_Ent_Cstar_Fin_II}
Assume $(A,\tau)=\lc{}M_{n}(\mathbb{C}),\tr\rc$ for $n\geq 2$ in $\mathbb{N}$. Note $M_{n-1}(\mathbb{C})\subset M_{n}(\mathbb{C})$ is non-unital. For all $k\in\lset{}1,\ldots,n\rset$, let $\lambda_{k}\in (0,1)$. Following Example \ref{BSP.Rel_Ent_Cstar_Fin_I}, the diagonal matrix $D:=\lc\lambda_{1},\ldots,\lambda_{n}\rc\in M_{n}(\mathbb{C})_{+}$ yields quantum relative entropy

\begin{align}\label{EQ.BSP.Rel_Ent_Cstar_Fin_II_1}
\Ent\hspace{-0.0375cm} \big(D^{\flat},I_{n}^{\flat}\big)=\sum_{k=1}^{n}\lambda_{k}\log\lambda_{k}.
\end{align}

\noindent We know $\big(\restr{0.925}{D}{M_{n-1}(\mathbb{C})}\big)^{\flat}=\lc\lambda_{1},\ldots,\lambda_{n-1}\rc^{\flat}\in M_{n-1}(\mathbb{C})_{+}$ and $\big(\restr{0.925}{I}{M_{n-1}(\mathbb{C})}\big)^{\flat}=I_{n-1}^{\flat}$. Moreover, we have $\lambda_{n}\log\lambda_{n}<0$ by hypothesis. Equation \ref{EQ.BSP.Rel_Ent_Cstar_Fin_II_1} lets us estimate

\begin{align}\label{EQ.BSP.Rel_Ent_Cstar_Fin_II_2}
\Ent\hspace{-0.0375cm} \big(D^{\flat},I_{n}^{\flat}\big)=\Ent\lc\lc\restr{0.925}{D}{M_{n-1}(\mathbb{C})}\rc^{\flat},I_{n-1}^{\flat}\rc{}+\lambda_{n}\log\lambda_{n}<\Ent\lc\lc\restr{0.925}{D}{M_{n-1}(\mathbb{C})}\rc^{\flat},I_{n-1}^{\flat}\rc{}<\infty.
\end{align}

\noindent Equation \ref{EQ.BSP.Rel_Ent_Cstar_Fin_II_2} shows Equation \ref{EQ.SSEC.L2W_Rel_Ent_AF_2} fails since $M_{n-1}(\mathbb{C})\subset M_{n}(\mathbb{C})$ is non-unital.
\end{bsp}


\subsubsection*{Extending to traces in the second variable}

Note Equation \ref{EQ.DFN.Rel_Ent_Cstar_1} does not let us take relative entropy w.r.t.~non-finite traces. We extend accordingly. Let $(A,\tau)$ be a tracial AF-$C^{*}$-algebra. Definition \ref{DFN.Rel_Ent_AF_Cstar_Trace} gives the relative entropy $\Enttau:A_{+}^{*}\longrightarrow [-\infty,\infty]$ w.r.t.~$\tau$, i.e.~quantum relative entropy. Proposition \ref{PRP.Rel_Ent_Cstar_II} shows Lemma \ref{LEM.Rel_Ent_AF_Cstar_Trace_II} recovers Equation \ref{EQ.DFN.Rel_Ent_Wstar_1} for normal positive bounded functionals with integrable support.

\begin{dfn}\label{DFN.Rel_Ent_AF_Cstar_Trace} 
Set extended trace norm $\|x\|_{U\lc\tau\rc{}}:=\sqrt{U\lc\tau\rc(x^{*}x)}$ for all $x\in U(A)$. For all $\mu\in A_{+}^{*}$, the relative entropy of $\mu$ w.r.t.~$\tau$ is defined by

\begin{align}\label{EQ.DFN.Rel_Ent_AF_Cstar_Trace_1}
\Ent(\mu,\tau):=\sup_{\substack{n\in\mathbb{N},\\ F\in\mathcal{T}_{n}(U(A))}}\left\{\|\mu\|_{A^{*}}\log n-\int_{n^{-1}}^{\infty}t^{-1}\dblv{}F^{\perp}(t)\dblv_{U(\mu)}^{2}+t^{-2}\dblv{}F(t)\dblv_{U\lc\tau\rc{}}^{2}dt\right\}.
\end{align}

\noindent Set $\Enttau:=\Ent\lc\hspace{0.025cm}.\hspace{0.025cm},\tau\rc{}:A_{+}^{*}\longrightarrow [-\infty,\infty]$ and $\dom\Enttau:=\lset\mu\in A_{+}^{*}\ \vset\ \absv{1.15}{\Ent(\mu,\tau)\hspace{0.025cm}}<\infty\rset$. We call $\Enttau$ quantum relative entropy w.r.t.~$\tau$, or quantum relative entropy.
\end{dfn}

\begin{ntn}\label{NTN.Rel_Ent_AF_Cstar_Trace}
For all $j\in\mathbb{N}$ and $\mu\in A_{j,+}^{*}$, let $\Ent\lc\mu,\tau_{j}\rc$ denote the relative entropy of $\mu$ w.r.t.~$\tau_{j}=\tau\vert_{A_{j}}$ for the tracial AF-$C^{*}$-algebra $(A_{j},\tau)=(A_{j},\tau_{j})$ as per Definition \ref{DFN.AF_Cstar_Trace_Restriction}.
\end{ntn}

\begin{rem}\label{REM.Rel_Ent_AF_Cstar_Trace_I}
For all $n\in\mathbb{N}$ and $F\in\mathcal{T}_{n}(U(A))$, traciality implies 

\begin{align}\label{EQ.REM.Rel_Ent_AF_Cstar_Trace_I_1}
\dblv{}F(t)\dblv_{U\lc\tau\rc{}}^{2}=U\lc\tau\rc\lc{}F(t)^{*}F(t)\rc{}=U\lc\tau\rc\lc{}F(t)F(t)^{*}\rc{}.
\end{align}

\noindent Note $\dblv{}F(t)\dblv_{U(\eta)}^{2}=U(\eta)\lc{}F(t)F(t)^{*}\rc$ in Equation \ref{EQ.DFN.Rel_Ent_Cstar_1}. Compare to Equation \ref{EQ.REM.Rel_Ent_AF_Cstar_Trace_I_1}, i.e.~use of extended trace norm, in Equation \ref{EQ.DFN.Rel_Ent_AF_Cstar_Trace_1}. If $\tau<\infty$, then Equation \ref{EQ.SSEC.L2W_Rel_Ent_AF_4} shows Equation \ref{EQ.DFN.Rel_Ent_AF_Cstar_Trace_1} is Equation \ref{EQ.DFN.Rel_Ent_Cstar_1} using $\eta=\tau$. If $\tau$ is non-finite, then its joint convexity implies $\Enttau$ is not l.s.c.~in $w^{*}$-topology of $A^{*}$ on weakly closed convex $K\subset A_{+}^{*}$ for which there exists $\mu\in K$ s.t.~$\Ent(\mu,\tau)=-\infty$. We argue as for Example 4.4 in \cite{ART.Stu.2006.Classical_OT_I}.
\end{rem}

For all $j\in\mathbb{N}$ and $\mu\in A_{j,+}^{*}$, using quantum relative entropy for $A_{j}$ yields

\begin{align}\label{EQ.SSEC.L2W_Rel_Ent_AF_5}
\Ent\lc\mu,\tau_{j}\rc{}=\sup_{\substack{n\in\mathbb{N},\\ F\in\mathcal{T}_{n}^{u}(A_{j})}}\left\{\|\mu\|_{A_{j}^{*}}\log n-\int_{n^{-1}}^{\infty}t^{-1}\dblv{}F^{\perp}(t)\dblv_{\mu}^{2}+t^{-2}\dblv{}F(t)\dblv_{\tau}^{2}dt\right\}.
\end{align}

\noindent For all $\mu\in A_{+}^{*}$, we expect $\Ent(\mu,\tau)=\lim_{j\in\mathbb{N}}\Ent\lc\mu_{j},\tau_{j}\rc$ if we indeed measure information required to discriminate $\mu$ and $\tau$ through observation. Theorem \ref{THM.Rel_Ent_AF_Cstar_Trace} shows this given uniform majorant of local support. The latter uses Lemma \ref{LEM.Rel_Ent_AF_Cstar_Trace_I} and Lemma \ref{LEM.Rel_Ent_AF_Cstar_Trace_II}.

\begin{prp}\label{PRP.Rel_Ent_AF_Cstar_Trace}
If $\mu\in L^{1}(A,\tau)_{+}^{\flat}$ s.t.~$\Ent(\mu,\tau)>-\infty$, then we have

\begin{align}\label{EQ.PRP.Rel_Ent_AF_Cstar_Trace_1}
\Ent(\mu,\tau)=\sup_{\substack{n\in\mathbb{N},\\ F\in\mathcal{T}_{n}^{u}(L^{2,\infty}(A,\tau))}}\left\{\|\mu\|_{A^{*}}\log n-\int_{n^{-1}}^{\infty}t^{-1}\dblv{}F^{\perp}(t)\dblv_{\mu}^{2}+t^{-2}\dblv{}F(t)\dblv_{\tau}^{2}dt\right\}.
\end{align}
\end{prp}
\begin{proof}
Let $\mu\in L^{1}(A,\tau)_{+}^{\flat}$ s.t.~$\Ent(\mu,\tau)>-\infty$. As for $1)$ in Proposition \ref{PRP.Rel_Ent_Cstar_II}, normality lets us drop $\varphi$ in Kosaki's formula while taking the supremum over all $\mathcal{T}_{n}^{u}\lc{}L^{\infty}(A,\tau)\rc$.\par


\pagebreak


Let $n\in\mathbb{N}$ and $F\in\mathcal{T}_{n}^{u}\lc{}L^{\infty}(A,\tau)\rc$. If there exists $t_{0}\in (n^{-1},\infty)$ s.t.~$F(t_{0})\notin L^{2}(A,\tau)$, then $F$ being a step function implies

\begin{align}\label{EQ.PRP.Rel_Ent_AF_Cstar_Trace_2}
\int_{n^{-1}}^{\infty}t^{-1}\dblv{}F^{\perp}(t)\dblv_{\mu}^{2}+t^{-2}\dblv{}F(t)\dblv_{\tau}^{2}dt\geq\int_{n^{-1}}^{\infty}t^{-2}\dblv{}F(t)\dblv_{\tau}^{2}dt=\infty.
\end{align}

\noindent Thus Equation \ref{EQ.PRP.Rel_Ent_AF_Cstar_Trace_2} and $\Ent(\mu,\tau)>-\infty$ imply we may restrict to $L^{2,\infty}(A,\tau)$, hence we see Equation \ref{EQ.SSEC.L2W_Rel_Ent_AF_4} shows Equation \ref{EQ.PRP.Rel_Ent_AF_Cstar_Trace_1}.
\end{proof}

\begin{lem}\label{LEM.Rel_Ent_AF_Cstar_Trace_I}
For all $j\in\mathbb{N}$ and $\mu\in A_{j,+}^{*}$, we have $\Ent(\mu,\tau)=\Ent\lc\mu,\tau_{j}\rc$.
\end{lem}
\begin{proof}
Let $j\in\mathbb{N}$ and $\mu\in A_{j,+}^{*}$. Note $\|\mu\|_{A_{j}^{*}}=\|\mu\|_{A^{*}}$. For all $x\in A_{j}$, we further know

\begin{align}\label{EQ.LEM.Rel_Ent_AF_Cstar_Trace_I_1}
\dblv{}1_{A_{j}}-x\dblv_{\mu}=\dblv{}1_{A}-x\dblv_{\mu}.
\end{align}

\noindent Lemma \ref{LEM.Rel_Ent_Cstar_Normality} shows $U(\mu)=\mu\circ\varphi$. Since $\varphi$ is a unital surjective $^{*}$-homomorphism, we map $\mathcal{T}_{n}^{u}(A_{j})$ to $\mathcal{T}_{n}^{u}(U(A))$ via $F\mapsto G:=\varphi^{-1}\circ F$ for all $n\in\mathbb{N}$ by choosing pre-images in each case. Equation \ref{EQ.LEM.Rel_Ent_AF_Cstar_Trace_I_1} and unitality show

\begin{align}\label{EQ.LEM.Rel_Ent_AF_Cstar_Trace_I_2}
\dblv{}1_{A_{j}}-F(t)\dblv_{\mu}=\dblv{}1_{A}-F(t)\dblv_{\mu}=\dblv{}G^{\perp}(t)\dblv_{U(\mu)}
\end{align}

\noindent in each case. Equation \ref{EQ.LEM.Rel_Ent_AF_Cstar_Trace_I_2} implies $\Ent\lc\mu,\tau_{j}\rc\leq\Ent(\mu,\tau)$ by Kosaki's formula. We show the converse. Since $\Ent\lc\mu,\tau_{j}\rc{}>-\infty$ by $4)$ in Proposition \ref{PRP.Rel_Ent_Cstar_I}, Proposition \ref{PRP.Rel_Ent_AF_Cstar_Trace} ensures we may use Equation \ref{EQ.PRP.Rel_Ent_AF_Cstar_Trace_1} as Kosaki's formula. For all $n\in\mathbb{N}$ and $F\in\mathcal{T}_{n}^{u}(L^{2,\infty}(A,\tau))$, set $F_{j}(t):=F(t)_{j}$ for all $t\in (n^{-1},\infty)$. Note $F_{j}\in\mathcal{T}_{n}^{u}(A_{j})$ in each case.\par
Let $n\in\mathbb{N}$, $F\in\mathcal{T}_{n}^{u}(L^{2,\infty}(A,\tau))$ and $t\in(n^{-1},\infty)$. Then $F_{j}(t)=\pi_{j}^{A}\lc{}F(t)\rc$ and $I-F_{j}(t)\in A_{j}^{\perp}$ by square integrability. We therefore have

\begin{align}\label{EQ.LEM.Rel_Ent_AF_Cstar_Trace_I_3}
\dblv{}F(t)\dblv_{\tau}^{2}\geq \dblv{}\pi_{j}^{A}\lc{}F(t)\rc\dblv_{\tau}^{2}=\dblv{}F_{j}(t)\dblv_{\tau}^{2}.
\end{align}

\noindent We use $A_{j}A_{j}^{\perp}=A_{j}^{\perp}A_{j}=0$. Proposition \ref{PRP.AF_Cstar_Trace_Dualisation_I} implies restriction maps commute with adjoining as they are positivity-preserving \lc{}cf.~Proposition \ref{PRP.PO_II}\rc{}. We calculate

\begin{align}\label{EQ.LEM.Rel_Ent_AF_Cstar_Trace_I_4}
\dblv{}F(t)\dblv_{\mu}^{2}=\dblv{}F_{j}(t)\dblv_{\mu}^{2}+\dblv{}\big(I-\pi_{j}^{A}\big)\lc{}F_{j}(t)\rc\dblv_{\mu}^{2}\geq \dblv{}F_{j}(t)\dblv_{\mu}^{2}.
\end{align}

\noindent Since $\mu\in A_{j}$ implies $\mu(u)=\mu(u_{j})$ for all $u\in L^{2}(A,\tau)$, multiplying out terms yields

\begin{align}\label{EQ.LEM.Rel_Ent_AF_Cstar_Trace_I_5}
\dblv{}1_{A}-F(t)\dblv_{\mu}^{2}=\mu(1_{A_{j}})-\mu\lc{}F_{j}(t)^{*}\rc{}-\mu\lc{}F_{j}(t)\rc{}+\dblv{}F(t)\dblv_{\mu}^{2}.
\end{align}

\noindent Note Equation \ref{EQ.LEM.Rel_Ent_AF_Cstar_Trace_I_4} lets us estimate the final summand in Equation \ref{EQ.LEM.Rel_Ent_AF_Cstar_Trace_I_5}. We moreover collect terms on the right-hand side of the resulting estimate. In summary, we obtain

\begin{align}\label{EQ.LEM.Rel_Ent_AF_Cstar_Trace_I_6}
\dblv{}F^{\perp}(t)\dblv_{\mu}^{2}=\dblv{}1_{A}-F(t)\dblv_{\mu}^{2}\geq \dblv{}1_{A_{j}}-F_{j}(t)\dblv_{\mu}^{2}=\dblv{}F_{j}^{\perp}(t)\dblv_{\mu}^{2}.
\end{align}


\pagebreak


For all $n\in\mathbb{N}$ and $F\in\mathcal{T}_{n}^{u}(L^{2,\infty}(A,\tau))$, applying Equation \ref{EQ.LEM.Rel_Ent_AF_Cstar_Trace_I_3} and Equation \ref{EQ.LEM.Rel_Ent_AF_Cstar_Trace_I_6} to integrands on the left-hand side below lets us estimate

\begin{align}\label{EQ.LEM.Rel_Ent_AF_Cstar_Trace_I_7}
\int_{n^{-1}}^{\infty}t^{-1}\dblv{}F^{\perp}(t)\dblv_{\mu}^{2}+t^{-2}\dblv{}F(t)\dblv_{\tau}^{2}dt\geq\int_{n^{-1}}^{\infty}t^{-1}\dblv{}F_{j}^{\perp}(t)\dblv_{\mu}^{2}+t^{-2}\dblv{}F_{j}(t)\dblv_{\tau}^{2}dt.
\end{align}

\noindent Using Equation \ref{EQ.SSEC.L2W_Rel_Ent_AF_5}, resp.~Equation \ref{EQ.PRP.Rel_Ent_AF_Cstar_Trace_1} as Kosaki's formula, Equation \ref{EQ.LEM.Rel_Ent_AF_Cstar_Trace_I_7} lets us estimate $\Ent\lc\mu,\tau_{j}\rc\geq\Ent(\mu,\tau)$. Altogether, get $\Ent(\mu,\tau)=\Ent\lc\mu,\tau_{j}\rc$.
\end{proof}

\begin{lem}\label{LEM.Rel_Ent_AF_Cstar_Trace_II}
Let $\mu\in A_{+}^{*}$.

\begin{itemize}
\item[1)] If $\mu\notin L^{1}(A,\tau)_{+}^{\flat}$, then $\mu\notin\dom\Enttau$.

\item[2)] If $\mu\in L^{1}(A,\tau)_{+}^{\flat}$ and $p\in L^{1,\infty}(A,\tau)$ is a projection s.t.~$\supp\mu\leq p$, then $\Ent(\mu,\tau)>-\infty$ and we have

\begin{align}\label{EQ.LEM.Rel_Ent_AF_Cstar_Trace_II_1}
\Ent(\mu,\tau)=\sup_{\substack{n\in\mathbb{N},\\ F\in\mathcal{T}_{n}^{u}(A[p])}}\left\{\|\mu\|_{A[p]^{*}}\log n-\int_{n^{-1}}^{\infty}t^{-1}\dblv{}F^{\perp}(t)\dblv_{\mu}^{2}+t^{-2}\dblv{}F(t)\dblv_{\tau}^{2}dt\right\}.
\end{align}

\begin{reapply}
\end{reapply}

\end{itemize}
\end{lem}
\begin{proof}
We show $1)$. Assume $\mu\notin L^{1}(A,\tau)_{+}^{\flat}$. If $\Ent(\mu,\tau)=-\infty$, then our claim follows at once. We assume $\Ent(\mu,\tau)>-\infty$ without loss of generality. Proposition \ref{PRP.Rel_Ent_AF_Cstar_Trace} ensures we may use Equation \ref{EQ.PRP.Rel_Ent_AF_Cstar_Trace_1} as Kosaki's formula. Using Equation \ref{EQ.PRP.Rel_Ent_AF_Cstar_Trace_1}, each step function is constant for sufficiently large $t>0$ and maps to $L^{2,\infty}(A,\tau)$. Since $\Ent(\mu,\tau)>-\infty$, there exist $x\in U(A)$ s.t.~$\dblv{}1_{A}-x\dblv_{U(\mu)}=0$ and $\varphi(x)\in L^{2,\infty}(A,\tau)$. Since $\mu\notin L^{1}(A,\tau)_{+}^{\flat}$, Lemma \ref{LEM.Rel_Ent_Cstar_Normality} yields projection $p\in U(A)$ s.t.~$U\lc\tau\rc\lc{}p\rc{}=0$ and $U(\mu)\lc{}p\rc{}>0$ holds, and Lemma \ref{LEM.Cstar_Trace_Abstract_Projection} shows $\|\mu\|_{A[p]^{*}}=\|\mu\|_{A^{*}}$. Let $C>0$ s.t.~$2C U(\mu)\lc{}p\rc{}>\|\mu\|_{A^{*}}$.\par
We require suitable sequence to estimate. For all $n\in\mathbb{N}$, set

\begin{align*}
F_{n}(t):=
\begin{cases}
Cp & \If\ t\in (n^{-1},n), \\
x & t\geq n.
\end{cases}
\end{align*}

\noindent Note $F_{n}\in\mathcal{T}_{n}^{u}(U(A))$ in each case. Selecting $F_{n}$ for all $n\in\mathbb{N}$, we estimate

\begin{align}\label{EQ.LEM.Rel_Ent_AF_Cstar_Trace_II_2}
\Ent(\mu,\tau)\geq\sup_{n\in\mathbb{N}}\hspace{0.025cm} \left\{\|\mu\|_{A^{*}}\log n-U(\mu)\lc{}1_{U(A)}-Cp\rc\int_{n^{-1}}^{n}t^{-1}dt-\big\|\varphi(x)\big\|_{\tau}^{2}\int_{n}^{\infty}t^{-2}dt\right\}.
\end{align}

\noindent Since $\int_{n^{-1}}^{n}t^{-1}dt=2\log n$ and $\int_{n}^{\infty}t^{-2}dt=n^{-1}$ for all $n\in\mathbb{N}$, Equation \ref{EQ.LEM.Rel_Ent_AF_Cstar_Trace_II_2} implies

\begin{align}\label{EQ.LEM.Rel_Ent_AF_Cstar_Trace_II_3}
\Ent(\mu,\tau)\geq\sup_{n\in\mathbb{N}}\hspace{0.025cm} \lc{}2C\bar\mu\lc{}p\rc{}-\|\mu\|_{A^{*}}\rc\cdot \log n-\big\|\varphi(x)\big\|_{\tau}^{2}\cdot n^{-1}=\infty.
\end{align}

\noindent Equation \ref{EQ.LEM.Rel_Ent_AF_Cstar_Trace_II_3} shows $\Ent(\mu,\tau)=\infty$ if $\Ent(\mu,\tau)>-\infty$. Altogether, get $1)$.\par
We show $2)$. Assume $\mu\in L^{1}(A,\tau)_{+}^{\flat}$. Let $p\in L^{1,\infty}(A,\tau)$ be a projection s.t.~$\supp\mu\leq p$. Lemma \ref{LEM.Support_Projection} therefore implies $\mu\in L^{1}(A[p],\tau)_{+}^{\flat}$ and $\sharp\mu=\sharp\mu\cdot p=p\cdot \sharp\mu$. Surjectivity of $\varphi$ yields element $x\in\varphi^{-1}\lc{}p\rc\in U(A)$. For all $n\in\mathbb{N}$, set $F_{n}(t):=x$ for all $t\in (n^{-1},\infty)$. Note $F_{n}\in\mathcal{T}_{n}^{u}(U(A))$ in each case. Selecting $F_{n}$ for all $n\in\mathbb{N}$, we estimate

\begin{align}\label{EQ.LEM.Rel_Ent_AF_Cstar_Trace_II_4}
\Ent(\mu,\tau)\geq\sup_{n\in\mathbb{N}}\hspace{0.025cm} \|\mu\|_{A^{*}}\log n-\int_{n^{-1}}^{\infty}t^{-1}\mu\lc{}1_{A}-p\rc{}+t^{-2}\tau\lc{}p\rc{}dt\geq-\tau\lc{}p\rc{}>-\infty.
\end{align}

\noindent Equation \ref{EQ.LEM.Rel_Ent_AF_Cstar_Trace_II_4} shows Proposition \ref{PRP.Rel_Ent_AF_Cstar_Trace} ensures we may use Equation \ref{EQ.PRP.Rel_Ent_AF_Cstar_Trace_1} as Kosaki's formula. We may furthermore take the supremum over all $\mathcal{T}_{n}^{u}\lc{}L^{\infty}(A,\tau)\rc$ \lc{}cf.~proof of Proposition \ref{PRP.Rel_Ent_AF_Cstar_Trace}\rc{}. Since $A[p]\subset L^{\infty}(A,\tau)$ is a $C^{*}$-subalgebra, we bound the variational expression on the right-hand side of Equation \ref{EQ.LEM.Rel_Ent_AF_Cstar_Trace_II_1} from above by $\Ent(\mu,\tau)$. We show the converse. For all $n\in\mathbb{N}$ and $F\in\mathcal{T}_{n}^{u}(L^{2,\infty}(A,\tau))$, set $F_{p}(t):=\comp F(t)=pF(t)p$ for all $t\in (n^{-1},\infty)$. Since $pAp\subset A[p]$ by definition, note $F_{p}\in\mathcal{T}_{n}^{u}(A[p])$.\par
Let $n\in\mathbb{N}$, $F\in\mathcal{T}_{n}^{u}(L^{2,\infty}(A,\tau))$ and $t\in (n^{-1},\infty)$. Using $\sharp\mu=\sharp\mu\cdot p=p\cdot \sharp\mu$, $p^{2}=p$ and traciality, we calculate

\begin{align}\label{EQ.LEM.Rel_Ent_AF_Cstar_Trace_II_5}
\dblv{}F^{\perp}(t)\dblv_{\mu}^{2}=\dblv{}p-F_{p}(t)\dblv_{\mu}^{2}+\mu\lc{}pF(t)^{*}\lc{}1_{A}-p\rc{}F(t)p\rc{}
\end{align}

\noindent and

\begin{align}\label{EQ.LEM.Rel_Ent_AF_Cstar_Trace_II_6}
\dblv{}F(t)\dblv_{\tau}^{2}=\dblv{}F_{p}(t)\dblv_{\tau}^{2}+\tau\lc{}pF(t)\lc{}1_{A}-p\rc{}F(t)^{*}\rc{}+\tau\lc\lc{}1_{A}-p\rc{}F(t)F(t)^{*}\rc{}.
\end{align}

\noindent We know $p,1_{A}-p\in L^{\infty}(A,\tau)_{+}$ by hypothesis. For all $y\in L^{\infty}(A,\tau)$, we therefore have $py^{*}\lc{}1-p\rc{}yp,y\lc{}1_{A}-p\rc{}y^{*},yy^{*}\in L^{\infty}(A,\tau)_{+}$. Using such positivity, Equation \ref{EQ.LEM.Rel_Ent_AF_Cstar_Trace_II_5} lets us estimate

\begin{align}\label{EQ.LEM.Rel_Ent_AF_Cstar_Trace_II_7}
\dblv{}F^{\perp}(t)\dblv_{\mu}^{2}\geq \dblv{}p-F_{p}(t)\dblv_{\mu}^{2},
\end{align}

\noindent and Equation \ref{EQ.LEM.Rel_Ent_AF_Cstar_Trace_II_6} lets us estimate

\begin{align}\label{EQ.LEM.Rel_Ent_AF_Cstar_Trace_II_8}
\dblv{}F(t)\dblv_{\tau}^{2}\geq \dblv{}F_{p}(t)\dblv_{\tau}^{2}.
\end{align}

We conclude by estimating integral terms in Kosaki's formula as follows. For all $n\in\mathbb{N}$ and $F\in\mathcal{T}_{n}^{u}(L^{2,\infty}(A,\tau))$, applying Equation \ref{EQ.LEM.Rel_Ent_AF_Cstar_Trace_II_7} and Equation \ref{EQ.LEM.Rel_Ent_AF_Cstar_Trace_II_8} to integrands on the left-hand side below lets us estimate

\begin{align}\label{EQ.LEM.Rel_Ent_AF_Cstar_Trace_II_9}
\int_{n^{-1}}^{\infty}t^{-1}\dblv{}F^{\perp}(t)\dblv_{\mu}^{2}+t^{-2}\dblv{}F(t)\dblv_{\tau}^{2}dt\geq\int_{n^{-1}}^{\infty}t^{-1}\dblv{}F_{p}^{\perp}(t)\dblv_{\mu}^{2}+t^{-2}\dblv{}F_{p}(t)\dblv_{\tau}^{2}dt.
\end{align}

\noindent Using Equation \ref{EQ.PRP.Rel_Ent_AF_Cstar_Trace_1} as Kosaki's formula, Equation \ref{EQ.LEM.Rel_Ent_AF_Cstar_Trace_II_9} shows we bound the variational expression on the right-hand side of Equation \ref{EQ.LEM.Rel_Ent_AF_Cstar_Trace_II_1} from below by $\Ent(\mu,\tau)$. Thus get Equation \ref{EQ.LEM.Rel_Ent_AF_Cstar_Trace_II_1}, hence using relative entropy for $A[p]$ yields $\Ent(\mu,\tau)=\Ent(\mu,p^{\flat})>-\infty$ by $4)$ in Proposition \ref{PRP.Rel_Ent_Cstar_I} and $1)$ in Proposition \ref{PRP.Rel_Ent_Cstar_II}. Altogether, get $2)$.
\end{proof}


\subsection{Restriction to finitely supported accessibility components}\label{SSEC.L2W_Rel_Ent_AC}

Theorem \ref{THM.Rel_Ent_AF_Cstar_Trace} shows compatibility of quantum relative entropy with compression and finite-dimensional approximation, as well as suitable l.s.c.~used in Theorem \ref{THM.QOT_Distance_AC_FS} to show l.s.c.~in topology of the given quantum optimal transport distance. We compress quantum relative entropy with uniform majorants of local support. Finite-dimensional approximation and l.s.c.~are given for states with uniform majorant of local support. As such, Theorem \ref{THM.Rel_Ent_AF_Cstar_Trace} shows we recover the strongly unital finite-trace case, and thereby the notion of discriminating information for quantum relative entropy as claimed in the introduction of this section, by compressing with uniform majorants of local support.\par
We further show all states in finitely supported accessibility components have such uniform majorants. Theorem \ref{THM.QOT_Distance_AC_FS} lets us restrict quantum relative entropy to each one s.t.~compatibility and l.s.c.~as above are satisfied. Assuming finitely supported fixed states, we are therefore able to apply the coarse graining process using Diagram \ref{EQ.SSEC.QOT_QIT_Encoding_16} for $K$ the domain of quantum relative entropy. In Section \ref{SEC.L2W_EVI}, we use the latter for our discussion, in particular our equivalence Theorem \ref{THM.L2W_EVI_Equivalence}. Examples of finitely supported fixed states arise from fixed states on tracial AF-$C^{*}$-algebras generating hyperfinite factors of type I and II by $\sigma$-weak closure.


\subsubsection*{Uniform majorants of local support}

Definition \ref{DFN.AF_Support_Projection_Majorant_Uniform} gives local support and uniform majorants of local support. Using the latter, we introduce finitely supported accessibility components. Strongly unital tracial AF-$C^{*}$-algebras with finite trace have units as uniform majorants of local support. We give examples for the non-unital and non-finite-trace case. Following Corollary \ref{COR.Rel_Ent_AF_Cstar_Trace}, Example \ref{BSP.AF_Support_Projection_Majorant_Uniform_Type_I} and Example \ref{BSP.AF_Support_Projection_Majorant_Uniform_Type_II_Infty} give examples of finitely supported fixed states.\par
Let $(A,\tau)$ be a tracial AF-$C^{*}$-algebra.

\begin{dfn}\label{DFN.AF_Support_Projection_Majorant_Uniform}
Let $p\in L^{1,\infty}(A,\tau)$ a projection.

\begin{itemize}
\item[1)] For all $\mu\in A_{+}^{*}$, we say that $p$ majorises the local support of $\mu$ if

\begin{align}\label{EQ.DFN.AF_Support_Projection_Majorant_Uniform_1}
\supp\mu_{j}\leq p    
\end{align}

\begin{reapply}
\end{reapply}

\noindent in $L^{\infty}(A,\tau)$ for a.e.~$j\in\mathbb{N}$. We further call $p$ a majorant of the local support of $\mu$ and write $\supp\mu\subset p$.
 
\item[2)] The set of local supports in $L^{\infty}(A,\tau)$ uniformly majorised by $p$ is defined by

\begin{align}\label{EQ.DFN.AF_Support_Projection_Majorant_Uniform_2}
\CI[p]:=\lset\mu\in A_{+}^{*}\ \vset\ \supp\mu\subset p\rset{}.   
\end{align}

\noindent We further call $p$ a uniform majorant of local support of all $\mu\in\CI[p]$.
\end{itemize}
\end{dfn}

\begin{rem}\label{REM.AF_Support_Projection_Majorant_Uniform}
If $\tau<\infty$, then $1_{A}$ majorises the local support of all $\mu\in A_{+}^{*}$.
\end{rem}

\begin{lem}\label{LEM.AF_Support_Projection_Majorant_Uniform}
Let $p\in L^{1,\infty}(A,\tau)$ be a projection. If $\mu\in L^{1}(A,\tau)_{+}^{\flat}\cap\CI[p]$, then $\supp\mu\leq p$ and $\mu$ has integrable support.
\end{lem}
\begin{proof}
Let $\mu\in L^{1}(A,\tau)_{+}^{\flat}\cap\CI[p]$. Let $j\in\mathbb{N}$. Using $2)$ in Lemma \ref{LEM.AF_Cstar_Bimodule_L1_SR}, applying $\comAj$ to Equation \ref{EQ.LEM.AF_Cstar_Bimodule_L1_SR_1} yields

\begin{align}\label{EQ.LEM.AF_Support_Projection_Majorant_Uniform_1}
S_{j}:=\comAj L_{\sharp\mu_{j}}=L_{\sharp\mu_{j},A_{j}}\leq T_{j}:=\comAj L_{\pi_{j}^{A}\big(\sqrt{\sharp\mu}\big){}}^{2}=L_{\pi_{j}^{A}\big(\sqrt{\sharp\mu}\big){},A_{j}}^{2}.
\end{align}

\noindent Equation \ref{EQ.LEM.AF_Support_Projection_Majorant_Uniform_1} shows $\ker T_{j}\subset\ker S_{j}$ and therefore

\begin{align}\label{EQ.LEM.AF_Support_Projection_Majorant_Uniform_2}
\pi_{\ker T_{j}}^{A}\leq\pi_{\ker S_{j}}^{A}.
\end{align}

\noindent Using $2)$ in Proposition \ref{PRP.Support_Projection_II}, Equation \ref{EQ.LEM.AF_Support_Projection_Majorant_Uniform_2} implies

\begin{align}\label{EQ.LEM.AF_Support_Projection_Majorant_Uniform_3}
\textrm{supp}_{A_{j}}^{c}\mu_{j}\leq\textrm{supp}_{A_{j}}^{c}\pi_{j}^{A}\vstretch{1.225}{\big(}\sqrt{\sharp\mu}\vstretch{1.225}{\big)}{}.
\end{align}

\noindent Applying $1)$ in Proposition \ref{PRP.Support_Projection_II} to Equation \ref{EQ.LEM.AF_Support_Projection_Majorant_Uniform_3} yields

\begin{align}\label{EQ.LEM.AF_Support_Projection_Majorant_Uniform_4}
\supp\pi_{j}^{A}\vstretch{1.225}{\big(}\sqrt{\sharp\mu}\vstretch{1.225}{\big)}\leq\supp\mu_{j}\leq p.
\end{align}

Note $1)$ in Proposition \ref{PRP.Support_Projection_II} shows $\supp\mu=\supp\sharp\mu=\chi_{(0,\infty]}\vstretch{1.225}{\big(}\sqrt{\sharp\mu}\vstretch{1.225}{\big)}$ by positivity and functional calculus. Thus $1)$ in Proposition \ref{PRP.Support_Projection_II} and $2)$ in Lemma \ref{LEM.AF_Support_Projection_II} show

\begin{align}\label{EQ.LEM.AF_Support_Projection_Majorant_Uniform_5}
\chi_{(0,\infty]}\vstretch{1.225}{\big(}\sqrt{\sharp\mu}\vstretch{1.225}{\big)}{}=\s\textrm{-}\lim_{j\in\mathbb{N}}\hspace{0.025cm} \chi_{(0,\infty]}\lc\pi_{j}^{A}\vstretch{1.225}{\big(}\sqrt{\sharp\mu}\vstretch{1.225}{\big)}\rc{}=\s\textrm{-}\lim_{j\in\mathbb{N}}\hspace{0.025cm} \supp\pi_{j}^{A}\vstretch{1.225}{\big(}\sqrt{\sharp\mu}\vstretch{1.225}{\big)}{},    
\end{align}

\noindent hence Equation \ref{EQ.LEM.AF_Support_Projection_Majorant_Uniform_4} and Equation \ref{EQ.LEM.AF_Support_Projection_Majorant_Uniform_5} lets us estimate

\begin{align}\label{EQ.LEM.AF_Support_Projection_Majorant_Uniform_6}
\supp\mu=\chi_{(0,\infty]}\vstretch{1.225}{\big(}\sqrt{\sharp\mu}\vstretch{1.225}{\big)}{}=\s\textrm{-}\lim_{j\in\mathbb{N}}\hspace{0.025cm} \supp\pi_{j}^{A}\vstretch{1.225}{\big(}\sqrt{\sharp\mu}\vstretch{1.225}{\big)}\leq p.
\end{align}

\noindent Equation \ref{EQ.LEM.AF_Support_Projection_Majorant_Uniform_6} shows $\supp\mu\leq p$. In particular, $\tau\lc\supp\mu\rc\leq\tau\lc{}p\rc{}<\infty$ as required.
\end{proof}

\begin{cor}\label{COR.AF_Support_Projection_Majorant_Uniform}
Let $p\in L^{1,\infty}(A,\tau)$ be a projection.

\begin{itemize}
\item[1)] $\CI[p]\subset A[p]_{+}^{*}$ and $\SII(A)\cap\CI[p]\subset\SII(A[p])$.

\item[2)] $L^{1}(A,\tau)_{+}^{\flat}\cap\CI[p]\subset L^{1}(A[p],\tau)_{+}^{\flat}$ and $\mathcal{S}^{\NI}(A)\cap\CI[p]\subset\mathcal{S}^{\NI}(A[p])$.

\item[3)] If $\mu\in\CI[p]$ and $\lset{}\mu^{k}\rset_{k\in K}\subset\CI[p]$ is a net s.t.~we have both $\mu=w^{*}$-$\lim_{k\in K}\mu^{k}$ in $A^{*}$ and $\|\mu\|_{A^{*}}=\lim_{k\in K}\dblv{}\mu^{k}\dblv_{A^{*}}$, then $\mu=w^{*}$-$\lim_{k\in K}\mu^{k}$ in $A[p]^{*}$.
\end{itemize}
\end{cor}
\begin{proof}
Lemma \ref{LEM.Cstar_Trace_Abstract_Projection} yields $A[p]_{+}^{*}\cap L^{1}(A[p],\tau)^{\flat}\subset A_{+}^{*}\cap L^{1}(A,\tau)^{\flat}$. Using normality, $2)$ in Proposition \ref{PRP.AF_Cstar_Unit} shows $\incp=\com_{p}^{*}:A[p]_{+}^{*}\cap L^{1}(A[p],\tau)^{\flat}\longrightarrow A_{+}^{*}\cap L^{1}(A,\tau)^{\flat}$ is injective. Thus Lemma \ref{LEM.AF_Support_Projection_Majorant_Uniform} implies $2)$, hence Proposition \ref{PRP.Cstar_Trace_Abstract_Dualisation_I} shows $3)$ at once. We show $1)$. For all $\mu\in\CI[p]$, $1)$ in Proposition \ref{PRP.AF_Cstar_Trace_Dualisation_II} implies $\mu=w^{*}$-$\lim_{j\in\mathbb{N}}\mu_{j}$ in $A^{*}$ and $\|\mu\|_{A^{*}}=\lim_{j\in\mathbb{N}}\|\mu_{j}\|_{A^{*}}$, and Lemma \ref{LEM.Support_Projection} shows $\lset\mu_{j}\rset_{j\in\mathbb{N}}\subset\CI[p]$ by scaling. Then $3)$ implies $1)$.
\end{proof}


\pagebreak


\begin{bsp}\label{BSP.AF_Support_Projection_Majorant_Uniform_Type_I}
Let $H$ be a separable Hilbert space. Assume $(A,\tau)=\lc\KII(H),\tr\rc$. Let $\mu\in S^{1}(H)_{+}^{\flat}$. Following Example \ref{BSP.AF_Support_Projection}, we know $\mu$ has integrable support if and only if $\supp\mu\in\KII(H)_{0}=\bigcup_{n\in\mathbb{N}}M_{n}(\mathbb{C})$. Lemma \ref{LEM.AF_Support_Projection_Majorant_Uniform} therefore implies $\mu$ has integrable support if and only if there exists a majorant of its local support.
\end{bsp}

\begin{bsp}\label{BSP.AF_Support_Projection_Majorant_Uniform_Type_II_Infty}
Let $H$ and $\HII$ be infinite-dimensional separable Hilbert spaces. We assume the setting of Example \ref{BSP.QOT_Type_II_Infty}, i.e.~assume $(A,\tau)=\lc\KII(H)\otimes\AII\lc\HII[J]\rc{},\tr\otimes\tau\rc$. Set $M:=L^{\infty}\lc\KII(H)\otimes\AII\lc\HII[J]\rc{},\tr\otimes\tau\rc$. Let $n\in\mathbb{N}$. We consider $\sigma$-weak closure 

\begin{align}\label{EQ.BSP.AF_Support_Projection_Majorant_Uniform_Type_II_Infty_1}
N:=\overline{M_{n}(\mathbb{C})\odot\AII(H[J])}\subset M.
\end{align}

\noindent Note $L^{1}\lc{}N,\tr\otimes\tau\rc{}=\overline{N}^{\|.\|_{1}}=\overline{M_{n}(\mathbb{C})\odot\AII(H[J])}^{\|.\|_{1}}$. Since $I_{n}\otimes 1_{\AII(H[J])}\in M_{n}(\mathbb{C})\odot\AII(H[J])$ is the unit, we have $1_{N}=I_{n}\otimes 1_{\AII(H[J])}$ by density in $\sigma$-weak topology. Thus $\lc\tr\otimes\tau\rc\lc{}1_{N}\rc{}=n<\infty$, hence $\tr\otimes\tau<\infty$ and therefore $N\subset \lc{}M,\tr\otimes\tau\rc$ \lc{}cf.~$1)$ in Proposition \ref{PRP.Wstar_Trace_Fin_I} and $1)$ in Proposition \ref{PRP.Wstar_Trace_NCE_II}\rc{}. We show $1_{N}$ majorises local support of all $\mu\in L^{1}\lc{}N,\tr\otimes\tau\rc_{+}^{\flat}$.\par
Let $\mu\in L^{1}\lc{}N,\tr\otimes\tau\rc_{+}^{\flat}$. Using separability to obtain sequences, Equation \ref{EQ.BSP.AF_Support_Projection_Majorant_Uniform_Type_II_Infty_1} yields $\{m_{k}\}_{k\in\mathbb{N}}\subset\mathbb{N}$ and $\lset{}x_{l}^{k}\otimes y_{l}^{k}\rset_{k,l\in\mathbb{N}}\subset M_{n}(\mathbb{C})\odot\AII(H[J])$ s.t.~$\sharp\mu=\|.\|_{1}$-$\lim_{k\in\mathbb{N}}\sum_{l=1}^{m_{k}}x_{l}^{k}\otimes y_{l}^{k}$. If $j\geq n$ in $\mathbb{N}$, then construction of $\tr\otimes\tau$ shows 

\begin{align}\label{EQ.BSP.AF_Support_Projection_Majorant_Uniform_Type_II_Infty_2}
\pi_{j}^{\KII(H)\otimes\AII\lc\HII[J]\rc{}}\big(M_{n}(\mathbb{C})\odot\AII(H[J])\big)\subset M_{n}(\mathbb{C})\odot\AII(H_{j}[J])\subset N.
\end{align}

\noindent Using $2.1)$ in Proposition \ref{PRP.AF_Cstar_Trace_Dualisation_I}, applying the flat operator to Equation \ref{EQ.BSP.AF_Support_Projection_Majorant_Uniform_Type_II_Infty_2} yields dual equivalent. For all $j\geq n$ in $\mathbb{N}$, $w^{*}$-continuity and linearity therefore imply

\begin{align}\label{EQ.BSP.AF_Support_Projection_Majorant_Uniform_Type_II_Infty_3}
\sharp\mu_{j}=\|.\|_{\infty}\textrm{-}\lim_{k\in\mathbb{N}}\hspace{0.025cm} \sum_{l=1}^{m_{k}}x_{l}^{k}\otimes y_{l,j}^{k}\in M_{n}(\mathbb{C})\odot\AII(H_{j}[J])\subset N.
\end{align}

\noindent Finite-dimensionality ensures a priori $\|.\|_{1}$-convergence for Equation \ref{EQ.BSP.AF_Support_Projection_Majorant_Uniform_Type_II_Infty_3} is equivalent to $\|.\|_{\infty}$-convergence as used. If $j\leq n$, then $\sharp\mu_{j}\in M_{j}(\mathbb{C})\odot\AII(H_{j}[J])\subset N$ shows $\supp\mu_{j}\in N$ by $1)$ in Proposition \ref{PRP.Support_Projection_II}. If instead $j\geq n$, then Equation \ref{EQ.BSP.AF_Support_Projection_Majorant_Uniform_Type_II_Infty_3} shows $\supp\mu_{j}\in N$ by said proposition. For all $j\in\mathbb{N}$, we therefore have $\supp\mu_{j}\leq 1_{N}$ since each is a projection.
\end{bsp}


\subsubsection*{Quantum relative entropy given uniform majorant of local support}

Using Lemma \ref{LEM.Rel_Ent_AF_Cstar_Trace_II}, Lemma \ref{LEM.AF_Support_Projection_Majorant_Uniform} and Lemma \ref{LEM.Rel_Ent_AF_Cstar_Trace_III}, Theorem \ref{THM.Rel_Ent_AF_Cstar_Trace} shows all properties we require of quantum relative entropy. We compress quantum relative entropy as per $1)$ in Theorem \ref{THM.Rel_Ent_AF_Cstar_Trace}. Finite-dimensional approximation is $3)$ in Theorem \ref{THM.Rel_Ent_AF_Cstar_Trace}. This is compatibility of quantum relative entropy with compression and finite-dimensional approximation. Its suitable l.s.c.~in topology of the given quantum optimal transport distance is $2)$ in Theorem \ref{THM.Rel_Ent_AF_Cstar_Trace}. Assuming boundedness, Corollary \ref{COR.Rel_Ent_AF_Cstar_Trace} gives closed term trace description of quantum relative entropy. Example \ref{BSP.Rel_Ent_Cstar_Fin_I} shows its negative is quantum entropy, i.e.~von Neumann entropy \lc{}cf.~p.17 in \cite{BK.Ohy_Pet.1993.Rel_Ent}\rc{}.\par


\pagebreak


Let $(A,\tau)$ be a tracial AF-$C^{*}$-algebra.

\begin{lem}\label{LEM.Rel_Ent_AF_Cstar_Trace_III}
Let $p\in L^{1,\infty}(A,\tau)$ be a projection.

\begin{itemize}
\item[1)] If $\mu\in\CI[p]$, then $\Ent(\mu,\tau)=\Ent\hspace{-0.0375cm} \big(\mu,p^{\flat}\big)>-\infty$ and we have

\begin{align}\label{EQ.LEM.Rel_Ent_AF_Cstar_Trace_III_1}
\Ent\hspace{-0.0375cm} \big(\mu,p^{\flat}\big)=\sup_{\substack{n\in\mathbb{N},\\ F\in\mathcal{T}_{n}^{u}(A[p])}}\left\{\|\mu\|_{A[p]^{*}}\log n-\int_{n^{-1}}^{\infty}t^{-1}\dblv{}F^{\perp}(t)\dblv_{\mu}^{2}+t^{-2}\dblv{}F(t)\dblv_{p^{\flat}}^{2}dt\right\}.
\end{align}

\begin{reapply}
\end{reapply}

\item[2)] If $\mu\in\CI[p]$ and $\lset\mu^{n}\rset_{n\in\mathbb{N}}\subset\CI[p]$ s.t.~we have both $\mu=w^{*}$-$\lim_{n\in\mathbb{N}}\mu^{n}$ in $A^{*}$ and $\|\mu\|_{A^{*}}=\lim_{n\in\mathbb{N}}\|\mu^{n}\|_{A^{*}}$, then

\begin{align}\label{EQ.LEM.Rel_Ent_AF_Cstar_Trace_III_2}
\Ent(\mu,\tau)\leq\liminf_{n\in\mathbb{N}}\hspace{0.025cm} \Ent\lc\mu^{n},\tau\rc{}.
\end{align}

\begin{reapply}
\end{reapply}

\end{itemize}
\end{lem}
\begin{proof}
Let $\mu\in\CI[p]$. We have tracial AF-$C^{*}$-algebra $(A[p],\tau)$ in $L^{\infty}(A[p],\tau)$ generated by $\lset{}A_{j}[p]\rset_{j\in\mathbb{N}}$. As such, we may apply our results concerning quantum relative entropy for tracial AF-$C^{*}$-algebras given in Subsection \ref{SSEC.L2W_Rel_Ent_AF}. We use unit $p$ in $A[p]$.\par
We show $1)$. Compression uses general $W^{*}$-algebras \lc{}cf.~Definition \ref{DFN.Compression_Abstract_Bd}\rc{}. Since $A[p]\subset A$, construction of universal enveloping $W^{*}$-algebras using $\sigma$-weak closure yields $W^{*}$-subalgebra $U(A[p])=\overline{\pi_{U}(A[p])}\subset U(A)$. Note $1_{U(A[p])}=\pi_{U}\lc{}p\rc$. Since $\varphi\circ\pi_{U}=\id_{A}$ extends to $\id_{L^{\infty}(A,\tau)}$ by normality, get $\varphi\lc{}1_{U(A[p])}\rc{}=p$. We therefore have

\begin{align}\label{EQ.LEM.Rel_Ent_AF_Cstar_Trace_III_3}
U(A[p])=U(A)\lb{}1_{U(A[p])}\rb\subset U(A). 
\end{align}

Get $\mu\in A[p]_{+}^{*}$ by $1)$ in Corollary \ref{COR.AF_Support_Projection_Majorant_Uniform}. Equation \ref{EQ.LEM.Rel_Ent_AF_Cstar_Trace_III_3} shows $U(\mu)\vert_{U(A[p])}$ is canonical normal extension of $\mu$ to $U(A[p])$. We know $p^{\flat}$ and $\tau$ coincide on $A[p]_{+}$. Equation \ref{EQ.LEM.Rel_Ent_AF_Cstar_Trace_III_3} shows $U\lc{}p^{\flat}\rc{}=U\lc\tau\rc\vert_{U(A[p])}$ as per Equation \ref{EQ.SSEC.L2W_Rel_Ent_AF_4}. We use relative entropy for $A$, resp.~$A[p]$ in Equation \ref{EQ.LEM.Rel_Ent_AF_Cstar_Trace_III_4} below. Equation \ref{EQ.LEM.Rel_Ent_AF_Cstar_Trace_III_3} lets us estimate

\begin{align}\label{EQ.LEM.Rel_Ent_AF_Cstar_Trace_III_4}
\Ent(\mu,\tau)\geq\Ent\hspace{-0.0375cm} \big(\mu,p^{\flat}\big).
\end{align}

\noindent If $\mu\notin\dom\Enttau$, then $4)$ in Proposition \ref{PRP.Rel_Ent_Cstar_I} and $1)$ in Lemma \ref{LEM.Rel_Ent_AF_Cstar_Trace_II} imply $\mu\notin L^{1}(A,\tau)_{+}^{\flat}$ and $\Ent(\mu,\tau)=\Ent(\mu,p^{\flat})=\infty$.  If $\mu\in\dom\Enttau$, then $2)$ in Lemma \ref{LEM.Rel_Ent_AF_Cstar_Trace_II} further implies $\mu\in L^{1}(A,\tau)_{+}^{\flat}$ and $\Ent(\mu,\tau)=\Ent(\mu,p^{\flat})>-\infty$. Get $1)$.\par
We show $2)$. Assume its setting. Note $3)$ in Corollary \ref{COR.AF_Support_Projection_Majorant_Uniform} shows $\mu=w^{*}$-$\lim_{n\in\mathbb{N}}\mu^{n}$ in $A[p]^{*}$. Thus l.s.c.~in Proposition \ref{PRP.Rel_Ent_Cstar_II} for $A[p]$ implies

\begin{align}\label{EQ.LEM.Rel_Ent_AF_Cstar_Trace_III_5}
\Ent\hspace{-0.0375cm} \big(\mu,p^{\flat}\big)\leq\liminf_{n\in\mathbb{N}}\hspace{0.025cm} \Ent\hspace{-0.0375cm} \big(\mu^{n},p^{\flat}\big),
\end{align}

\noindent hence $1)$ and Equation \ref{EQ.LEM.Rel_Ent_AF_Cstar_Trace_III_5} imply Equation \ref{EQ.LEM.Rel_Ent_AF_Cstar_Trace_III_2}.
\end{proof}

\begin{thm}\label{THM.Rel_Ent_AF_Cstar_Trace}
Let $(A,\tau)$ be a tracial AF-$C^{*}$-algebra. Let $p\in L^{1,\infty}(A,\tau)$ be a projection.

\begin{itemize}
\item[1)] If $\mu\in\SII(A)\cap\CI[p]$, then $\Ent(\mu,\tau)=\Ent\hspace{-0.0375cm} \big(\mu,p^{\flat}\big)>-\infty$ and we have

\begin{align}\label{EQ.THM.Rel_Ent_AF_Cstar_Trace_1}
\Ent\hspace{-0.0375cm} \big(\mu,p^{\flat}\big)=\sup_{\substack{n\in\mathbb{N},\\ F\in\mathcal{T}_{n}^{u}(A[p])}}\left\{\|\mu\|_{A[p]^{*}}\log n-\int_{n^{-1}}^{\infty}t^{-1}\dblv{}F^{\perp}(t)\dblv_{\mu}^{2}+t^{-2}\dblv{}F(t)\dblv_{p^{\flat}}^{2}dt\right\}.
\end{align}

\begin{reapply}
\end{reapply}

\item[2)] If $\mu\in\SII(A)\cap\CI[p]$ and $\lset\mu^{n}\rset_{n\in\mathbb{N}}\subset\SII(A)\cap\CI[p]$ s.t.~$\mu=w^{*}$-$\lim_{n\in\mathbb{N}}\mu^{n}$, then

\begin{align}\label{EQ.THM.Rel_Ent_AF_Cstar_Trace_2}
\Ent(\mu,\tau)\leq\liminf_{n\in\mathbb{N}}\hspace{0.025cm} \Ent\lc\mu^{n},\tau\rc{}.
\end{align}

\begin{reapply}
\end{reapply}

\item[3)] If $\mu\in\SII(A)\cap\CI[p]$, then

\begin{align}\label{EQ.THM.Rel_Ent_AF_Cstar_Trace_3}
\Ent(\mu,\tau)=\lim_{j\in\mathbb{N}}\hspace{0.025cm} \Ent\lc\mu_{j},\tau\rc{}=\lim_{j\in\mathbb{N}}\hspace{0.025cm} \Ent\lc\bar{\mu}_{j},\tau\rc{}.
\end{align}

\begin{reapply}
\end{reapply}

\end{itemize}
\end{thm}
\begin{proof}
Let $\mu\in\SII(A)\cap\CI[p]$. Lemma \ref{LEM.Rel_Ent_AF_Cstar_Trace_III} shows $1)$ and $2)$ at once. Since $\mu=w^{*}$-$\lim_{j\in\mathbb{N}}\mu_{j}$ by $1.2)$ in Proposition \ref{PRP.AF_Cstar_Trace_Dualisation_II}, Equation \ref{EQ.THM.Rel_Ent_AF_Cstar_Trace_3} follows from $2)$ if

\begin{align}\label{EQ.THM.Rel_Ent_AF_Cstar_Trace_4}
\Ent(\mu,\tau)\geq\liminf_{j\in\mathbb{N}}\hspace{0.025cm} \Ent\lc\mu_{j},\tau\rc{}.
\end{align}

\noindent Note $1.1)$ in Proposition \ref{PRP.AF_Cstar_Trace_Dualisation_II} and Proposition \ref{PRP.Rel_Ent_Cstar_I} ensure scaling upon restriction is of no consequence in Equation \ref{EQ.THM.Rel_Ent_AF_Cstar_Trace_3}. We show Equation \ref{EQ.THM.Rel_Ent_AF_Cstar_Trace_4}. For this, we engage in several reduction steps. If $\mu\notin L^{1}(A,\tau)_{+}^{\flat}$, then $\mu\notin\dom\Enttau$ by $1)$ in Lemma \ref{LEM.Rel_Ent_AF_Cstar_Trace_II}. Thus $\Ent(\mu,\tau)>-\infty$ further implies $\Ent(\mu,\tau)=\infty$, hence Equation \ref{EQ.THM.Rel_Ent_AF_Cstar_Trace_4} as well. We assume $\mu\in L^{1}(A,\tau)_{+}^{\flat}$ without loss of generality. We use the following. Lemma \ref{LEM.AF_Support_Projection_Majorant_Uniform} shows $\mu$ has integrable support. Theorem \ref{THM.AF_Support_Projection} ensures $\mu$ has reducible support.\par
We engage in the first reduction. We define remainder terms in Equation \ref{EQ.THM.Rel_Ent_AF_Cstar_Trace_7} below and show the latter implies Equation \ref{EQ.THM.Rel_Ent_AF_Cstar_Trace_4}. For all $j\in\mathbb{N}$, we consider the $C^{*}$-subalgebra $\mathcal{A}_{j}:=\vstretch{0.8575}{\big\langle} 1_{A_{j}}^{\perp}\vstretch{0.8575}{\big\rangle}_{\mathbb{C}}\subset L^{\infty}(A,\tau)$ and set

\begin{align}\label{EQ.THM.Rel_Ent_AF_Cstar_Trace_5}
\mu_{j}^{\perp}:=\mu\vert_{\mathcal{A}_{j}},\nu_{j}:=\restr{0.925}{\lc\supp\mu\rc^{\flat}}{\mathcal{A}_{j}}\in\mathcal{A}_{j,+}^{*}.
\end{align}

\noindent We further define the $j$-th remainder term $\Rj\in\mathbb{R}$ by setting

\begin{align*}
\Rj(\mu):=
\begin{cases}
\Ent\big(\mu_{j}^{\perp},\nu_{j}\big) & \If\ 1_{A_{j}}^{\perp}\neq 0, \\
0 & \Else.
\end{cases}
\end{align*}

Using $1_{A}=\s$-$\lim_{j\in\mathbb{N}}1_{A_{j}}$ as per $2)$ in Proposition \ref{PRP.AF_Cstar_Unit}, Example \ref{BSP.Rel_Ent_Cstar_Fin_I} for $n=1$ shows we either have $\Rj(\mu)=0$ for a.e.~$j\in\mathbb{N}$ or normality and positivity let us estimate

\begin{align}\label{EQ.THM.Rel_Ent_AF_Cstar_Trace_6}
\liminf_{j\in\mathbb{N}}\hspace{0.025cm} \Rj(\mu)=\liminf_{j\in\mathbb{N}}\hspace{0.025cm} \mu\big(1_{A_{j}}^{\perp}\big)\log\lc\mu\big(1_{A_{j}}^{\perp}\big)\rc{}-\mu\big(1_{A_{j}}^{\perp}\big)\log\lc\tau\big(\supp\mu\cdot 1_{A_{j}}^{\perp}\big)\rc\geq 0.
\end{align}

\noindent For details on our estimate of Equation \ref{EQ.THM.Rel_Ent_AF_Cstar_Trace_6}, we refer to Remark \ref{REM.Rel_Ent_AF_Cstar_Trace_II}. Equation \ref{EQ.THM.Rel_Ent_AF_Cstar_Trace_6} shows $\liminf_{j\in\mathbb{N}}\Rj(\mu)\geq 0$. We claim Equation \ref{EQ.THM.Rel_Ent_AF_Cstar_Trace_4} follows if

\begin{align}\label{EQ.THM.Rel_Ent_AF_Cstar_Trace_7}
\Ent(\mu,\tau)\geq\Ent\lc\mu_{j},\tau\rc{}+\Rj(\mu)
\end{align}

\noindent for all $j\in\mathbb{N}$. If we do have Equation \ref{EQ.THM.Rel_Ent_AF_Cstar_Trace_7}, then we apply $\liminf_{j\in\mathbb{N}}\Rj(\mu)\geq 0$ to estimate $\Ent(\mu,\tau)\geq\liminf_{j\in\mathbb{N}}\Ent\lc\mu_{j},\tau\rc{}+\Rj(\mu)\geq\liminf_{j\in\mathbb{N}}\Ent\lc\mu_{j},\tau\rc$. Thus Equation \ref{EQ.THM.Rel_Ent_AF_Cstar_Trace_7} implies Equation \ref{EQ.THM.Rel_Ent_AF_Cstar_Trace_4}, hence it suffices to show the former.\par
We engage in the second reduction. Let $j\in\mathbb{N}$. Note Equation \ref{EQ.THM.Rel_Ent_AF_Cstar_Trace_7} follows if

\begin{align}\label{EQ.THM.Rel_Ent_AF_Cstar_Trace_8}
\Ent(\mu,\tau)\geq\Ent\lc\mu_{j},\lc\supp\mu\rc_{j}^{\flat}\rc{}+\Rj(\mu),\ \Ent\lc\mu_{j},\lc\supp\mu\rc_{j}^{\flat}\rc\geq\Ent\lc\mu_{j},\tau\rc{}.
\end{align}

\noindent Set $\eta_{j}:=\lc\supp\mu\rc_{j}^{\flat}$ and $\nu_{j}:=\lc\supp\mu_{j}\rc^{\flat}$. We show Equation \ref{EQ.THM.Rel_Ent_AF_Cstar_Trace_8}. For this, we show

\begin{align}\label{EQ.THM.Rel_Ent_AF_Cstar_Trace_9}
\Ent\lc\mu_{j},\eta_{j}\rc\geq\Ent\lc\mu_{j},\tau\rc{}.
\end{align}

\noindent Applying Lemma \ref{LEM.Rel_Ent_AF_Cstar_Trace_I} and $1)$ to $\mu_{j}$ for $p=\supp\mu_{j}$ shows

\begin{align}\label{EQ.THM.Rel_Ent_AF_Cstar_Trace_10}
 \Ent\lc\mu_{j},\tau\rc{}=\Ent\lc\mu_{j},\tau_{j}\rc{}=\Ent\lc\mu_{j},\nu_{j}\rc{}.
\end{align}

\noindent Using Equation \ref{EQ.SSEC.L2W_Rel_Ent_AF_5} as Kosaki's formula, Equation \ref{EQ.THM.Rel_Ent_AF_Cstar_Trace_10} implies Equation \ref{EQ.THM.Rel_Ent_AF_Cstar_Trace_9} if $\nu_{j}\geq\eta_{j}$ in $A_{j}^{*}$. Following $2)$ in Proposition \ref{PRP.Support_Projection_II}, we equivalently estimate

\begin{align}\label{EQ.THM.Rel_Ent_AF_Cstar_Trace_11}
x_{j}:=1_{A_{j}}-\sharp\eta_{j}\geq 1_{A_{j}}-\supp\mu_{j}=L_{A_{j}}^{-1}\lc\pi_{\ker L_{\sharp\mu_{j},A_{j}}}^{A}\rc{}
\end{align}

\noindent in $A_{j}$. Note $x_{j}=\pi_{j}^{A}\lc{}1_{A}-\supp\mu\rc\geq 0$ by Proposition \ref{PRP.AF_Cstar_Trace_Dualisation_I}. Since $A_{0}\subset\dom L_{\sqrt{\sharp\mu}}$, get

\begin{align}\label{EQ.THM.Rel_Ent_AF_Cstar_Trace_12}
\dblv{}\sqrt{\sharp\mu}\cdot u\dblv_{\tau}^{2}=\lgl\sharp\mu_{j}u,u\rgl_{\tau}=0
\end{align}

\noindent for all $u\in\ker L_{\sharp\mu_{j},A_{j}}$. Equation \ref{EQ.THM.Rel_Ent_AF_Cstar_Trace_12} shows

\begin{align}\label{EQ.THM.Rel_Ent_AF_Cstar_Trace_13}
\ker L_{\sharp\mu_{j},A_{j}}\subset\ker L_{\sqrt{\sharp\mu}}.
\end{align}

Note $1)$ in Proposition \ref{PRP.Support_Projection_II} shows $\supp\mu=\supp\sharp\mu=\chi_{(0,\infty]}\vstretch{1.225}{\big(}\sqrt{\sharp\mu}\vstretch{1.225}{\big)}$ by positivity and functional calculus. Thus $1)$ and $2)$ in Proposition \ref{PRP.Support_Projection_II} imply

\begin{align}\label{EQ.THM.Rel_Ent_AF_Cstar_Trace_14}
\supp\mu=\chi_{(0,\infty]}\vstretch{1.225}{\big(}\sqrt{\sharp\mu}\vstretch{1.225}{\big)}{}=L^{-1}\lc{}I-\pi_{\ker\sqrt{\sharp\mu}}^{A}\rc{},
\end{align}

\noindent hence Equation \ref{EQ.THM.Rel_Ent_AF_Cstar_Trace_13} shows

\begin{align}\label{EQ.THM.Rel_Ent_AF_Cstar_Trace_15}
\dblv{}\sqrt{\eta_{j}}\cdot u\dblv_{\tau}^{2}=\lgl\supp\mu\cdot u,u\rgl_{\tau}=0
\end{align}

\noindent for all $u\in\ker L_{\sharp\mu_{j},A_{j}}$. Equation \ref{EQ.THM.Rel_Ent_AF_Cstar_Trace_15} shows

\begin{align}\label{EQ.THM.Rel_Ent_AF_Cstar_Trace_16}
\ker L_{\sharp\mu_{j},A_{j}}\subset\ker L_{\sharp\eta_{j},A_{j}}.
\end{align}

For all $u\in A_{j}$, we decompose as per Equation \ref{EQ.THM.Rel_Ent_AF_Cstar_Trace_17} below using the following. For the left-hand side of Equation \ref{EQ.THM.Rel_Ent_AF_Cstar_Trace_17}, apply $1)$ and $2)$ in Proposition \ref{PRP.Support_Projection_II}. For the right-hand side of Equation \ref{EQ.THM.Rel_Ent_AF_Cstar_Trace_17}, we use Equation \ref{EQ.THM.Rel_Ent_AF_Cstar_Trace_16}. Altogether, we have

\begin{align}\label{EQ.THM.Rel_Ent_AF_Cstar_Trace_17}
u=\supp\mu_{j}\cdot u+\pi_{\ker L_{\sharp\mu_{j},A_{j}}}^{A}(u),\ x_{j}\pi_{\ker L_{\sharp\mu_{j},A_{j}}}^{A}(u)=\pi_{\ker L_{\sharp\mu_{j},A_{j}}}^{A}(u)
\end{align}

\noindent for all $u\in A_{j}$. Equation \ref{EQ.THM.Rel_Ent_AF_Cstar_Trace_17} lets us calculate

\begin{align}\label{EQ.THM.Rel_Ent_AF_Cstar_Trace_18}
\lgl x_{j}u,u\rgl_{\tau}=\lgl x_{j}\supp\mu_{j}\cdot u,\supp\mu_{j}\cdot u\rgl_{\tau}+\lgl\pi_{\ker L_{\sharp\mu_{j},A_{j}}}^{A}(u),u\rgl_{\tau}
\end{align}

\noindent for all $u\in A_{j}$. As $x_{j}\geq 0$ yields $\lgl x_{j}\supp\mu_{j}\cdot u,\supp\mu_{j}\cdot u\rgl_{\tau}\geq 0$ in each case, Equation \ref{EQ.THM.Rel_Ent_AF_Cstar_Trace_18} lets us estimate

\begin{align}\label{EQ.THM.Rel_Ent_AF_Cstar_Trace_19}
\lgl x_{j}u,u\rgl_{\tau}\geq\lgl\pi_{\ker L_{\sharp\mu_{j},A_{j}}}^{A}(u),u\rgl_{\tau}
\end{align}

\noindent for all $u\in A_{j}$. Equation \ref{EQ.THM.Rel_Ent_AF_Cstar_Trace_19} shows Equation \ref{EQ.THM.Rel_Ent_AF_Cstar_Trace_11}. Using the latter, Equation \ref{EQ.THM.Rel_Ent_AF_Cstar_Trace_10} then implies Equation \ref{EQ.THM.Rel_Ent_AF_Cstar_Trace_9} as discussed above.\par
We engage in the third reduction. Following Equation \ref{EQ.THM.Rel_Ent_AF_Cstar_Trace_9}, we show

\begin{align}\label{EQ.THM.Rel_Ent_AF_Cstar_Trace_20}
\Ent(\mu,\tau)\geq\Ent\lc\mu_{j},\eta_{j}\rc{}+\Rj(\mu)
\end{align}

\noindent in order to have Equation \ref{EQ.THM.Rel_Ent_AF_Cstar_Trace_8} and therefore $3)$ as discussed above. Set $\eta:=\lc\supp\mu\rc^{\flat}$. Note $2)$ in Proposition \ref{PRP.Rel_Ent_Cstar_II} for $N=A_{j}[1_{A}]\subset L^{\infty}(A,\tau)$ and $1)$ applied to $\mu\in L^{1}(A,\tau)_{+}^{\flat}$ for $p=\supp\mu$ lets us estimate

\begin{align}\label{EQ.THM.Rel_Ent_AF_Cstar_Trace_21}
\Ent(\mu,\tau)=\Ent(\mu,\eta)\geq\Ent\lc\mu\vert_{A_{j}[1_{A}]},\eta\vert_{A_{j}[1_{A}]}\rc{}.
\end{align}

Note $\mathcal{A}_{j}=\vstretch{0.8575}{\big\langle} 1_{A_{j}}^{\perp}\vstretch{0.8575}{\big\rangle}_{\mathbb{C}}\subset L^{\infty}(A,\tau)$. Equation \ref{EQ.THM.Rel_Ent_AF_Cstar_Trace_21} implies Equation \ref{EQ.THM.Rel_Ent_AF_Cstar_Trace_20} if

\begin{align}\label{EQ.THM.Rel_Ent_AF_Cstar_Trace_22}
\Ent\lc\mu\vert_{A_{j}[1_{A}]},\eta\vert_{A_{j}[1_{A}]}\rc{}=\Ent\lc\mu_{j},\eta_{j}\rc{}+\Rj(\mu).    
\end{align}

\noindent If $1_{A_{j}}^{\perp}=0$, then Equation \ref{EQ.THM.Rel_Ent_AF_Cstar_Trace_22} holds since $\mathcal{A}_{j}=0$ and $\Rj(\mu)=0$. Assume $1_{A_{j}}^{\perp}\neq 0$. Get

\begin{align}\label{EQ.THM.Rel_Ent_AF_Cstar_Trace_23}
A_{j}[1_{A}]=A_{j}\oplus\mathcal{A}_{j}
\end{align}

\noindent by hypothesis \lc{}cf.~Proposition \ref{PRP.Cstar_Unitalisation}\rc{}. For all $\nu\in A_{j}[1_{A}]_{+}^{*}$, we decompose its norm

\begin{align}\label{EQ.THM.Rel_Ent_AF_Cstar_Trace_24}
\|\nu\|_{A_{j}[1_{A}]^{*}}=\|\nu\|_{A_{j}^{*}}+\|\nu\|_{\mathcal{A}_{j}^{*}}
\end{align}

\noindent over the direct sum of $C^{*}$-algebras as per Equation \ref{EQ.THM.Rel_Ent_AF_Cstar_Trace_23}. For all $n\in\mathbb{N}$, decomposing as per Equation \ref{EQ.THM.Rel_Ent_AF_Cstar_Trace_24} at each time yields further product decomposition

\begin{align}\label{EQ.THM.Rel_Ent_AF_Cstar_Trace_25}
\mathcal{T}_{n}^{u}\lc{}A_{j}[1_{A}]\rc{}=\mathcal{T}_{n}^{u}(A_{j})\times\mathcal{T}_{n}^{u}(\mathcal{A}_{j}).
\end{align}

\noindent Using Equation \ref{EQ.THM.Rel_Ent_AF_Cstar_Trace_1} as Kosaki's formula, Equation \ref{EQ.THM.Rel_Ent_AF_Cstar_Trace_25} implies

\begin{align}\label{EQ.THM.Rel_Ent_AF_Cstar_Trace_26}
\Ent\lc\mu\vert_{A_{j}[1_{A}]},\eta\vert_{A_{j}[1_{A}]}\rc{}=\Ent\lc\mu_{j},\eta_{j}\rc{}+\Ent\lc\mu\vert_{\mathcal{A}_{j}},\restr{0.925}{\lc\supp\mu\rc^{\flat}}{\mathcal{A}_{j}}\rc{}.
\end{align}

\noindent The second summand on the right-hand side of Equation \ref{EQ.THM.Rel_Ent_AF_Cstar_Trace_26} is $\Rj(\mu)$. Equation \ref{EQ.THM.Rel_Ent_AF_Cstar_Trace_22} holds. Equation \ref{EQ.THM.Rel_Ent_AF_Cstar_Trace_20} and therefore $3)$ follows as discussed above.
\end{proof}

\begin{rem}\label{REM.Rel_Ent_AF_Cstar_Trace_II}
We elaborate on our estimate of Equation \ref{EQ.THM.Rel_Ent_AF_Cstar_Trace_6}. We have

\begin{align}\label{EQ.REM.Rel_Ent_AF_Cstar_Trace_II_1}
\lim_{j\in\mathbb{N}}\hspace{0.025cm} \mu\big(1_{A_{j}}^{\perp}\big)=\lim_{j\in\mathbb{N}}\hspace{0.025cm} \tau\big(\supp\mu\cdot 1_{A_{j}}^{\perp}\big)=0,\ \mu\big(1_{A_{j}}^{\perp}\big),\tau\big(\supp\mu\cdot 1_{A_{j}}^{\perp}\big)\geq 0
\end{align}

\noindent by normality, resp.~for all $j\in\mathbb{N}$ by positivity. Using $\lim_{\lambda\rightarrow 0}\lambda\log\lambda=0$, Equation \ref{EQ.REM.Rel_Ent_AF_Cstar_Trace_II_1} lets us estimate

\begin{align*}
\liminf_{j\in\mathbb{N}}\hspace{0.025cm} \Rj(\mu) & = \liminf_{j\in\mathbb{N}}\hspace{0.025cm} \mu\big(1_{A_{j}}^{\perp}\big)\log\lc\mu\big(1_{A_{j}}^{\perp}\big)\rc{}-\mu\big(1_{A_{j}}^{\perp}\big)\log\lc\tau\big(\supp\mu\cdot 1_{A_{j}}^{\perp}\big)\rc \phantom{\Bigg)} \\
& \geq\liminf_{j\in\mathbb{N}}\hspace{0.025cm} -\mu\big(1_{A_{j}}^{\perp}\big)\log\lc\tau\big(\supp\mu\cdot 1_{A_{j}}^{\perp}\big)\rc \phantom{\Bigg)} \\
& \geq 0 \phantom{\Bigg)}
\end{align*}

\noindent since $\lim_{j\in\mathbb{N}}\log\bigg(\tau\big(\supp\mu\cdot 1_{A_{j}}^{\perp}\big)\bigg)=-\infty$ by normality.
\end{rem}

\begin{cor}\label{COR.Rel_Ent_AF_Cstar_Trace}
Let $p\in L^{1,\infty}(A,\tau)$ be a projection. If $\mu\in L^{1,\infty}(A,\tau)_{+}^{\flat}\cap\CI[p]$, then $\Ent(\mu,\tau)\in (0,\infty)$ and we have

\begin{align}\label{EQ.COR.Rel_Ent_AF_Cstar_Trace_1}
\Ent(\mu,\tau)=\tau\lc\sharp\mu\log\sharp\mu\rc{}=\tau\lc\comp\sharp\mu\log\comp\sharp\mu\rc{}.
\end{align}
\end{cor}
\begin{proof}
Let $\mu\in L^{1,\infty}(A,\tau)_{+}^{\flat}\cap\CI[p]$. We have $\comp\sharp\mu=p\cdot \sharp\mu\cdot p\in L^{\infty}(A[p],\tau)$. We use the following. Lemma \ref{LEM.AF_Support_Projection_Majorant_Uniform} shows $\mu$ has integrable support. Theorem \ref{THM.AF_Support_Projection} ensures $\mu$ has reducible support. Using Lemma \ref{LEM.Rel_Ent_AF_Cstar_Trace_I} and $3)$ in Theorem \ref{THM.Rel_Ent_AF_Cstar_Trace}, we calculate

\begin{align}\label{EQ.COR.Rel_Ent_AF_Cstar_Trace_2}
\Ent(\mu,\tau)=\lim_{j\in\mathbb{N}}\hspace{0.025cm} \Ent\lc\mu_{j},\tau\rc{}=\lim_{j\in\mathbb{N}}\hspace{0.025cm} \Ent\lc\mu_{j},\tau_{j}\rc{}.
\end{align}

\noindent Reduction using Proposition \ref{PRP.AF_Cstar_Trace_II} in Example \ref{BSP.Rel_Ent_Cstar_Fin_I} yields

\begin{align}\label{EQ.COR.Rel_Ent_AF_Cstar_Trace_3}
\Ent\lc\mu_{j},\tau_{j}\rc{}=\tau\lc\sharp\mu_{j}\log\sharp\mu_{j}\rc{}
\end{align}

\noindent for all $j\in\mathbb{N}$. Equation \ref{EQ.COR.Rel_Ent_AF_Cstar_Trace_2} and Equation \ref{EQ.COR.Rel_Ent_AF_Cstar_Trace_3} imply

\begin{align}\label{EQ.COR.Rel_Ent_AF_Cstar_Trace_4}
\Ent(\mu,\tau)=\lim_{j\in\mathbb{N}}\hspace{0.025cm} \tau\lc\sharp\mu_{j}\log\sharp\mu_{j}\rc{}.
\end{align}

\noindent Note Equation \ref{EQ.COR.Rel_Ent_AF_Cstar_Trace_4} shows Equation \ref{EQ.COR.Rel_Ent_AF_Cstar_Trace_1} if

\begin{align}\label{EQ.COR.Rel_Ent_AF_Cstar_Trace_5}
\tau\lc\sharp\mu\log\sharp\mu\rc{}=\lim_{j\in\mathbb{N}}\hspace{0.025cm} \tau\lc\sharp\mu_{j}\log\sharp\mu_{j}\rc{}
\end{align}

\noindent and further

\begin{align}\label{EQ.COR.Rel_Ent_AF_Cstar_Trace_6}
\tau\lc\comp\sharp\mu\log\comp\sharp\mu\rc{}=\lim_{j\in\mathbb{N}}\hspace{0.025cm} \tau\lc\comp\sharp\mu_{j}\log\comp\sharp\mu_{j}\rc{}=\lim_{j\in\mathbb{N}}\hspace{0.025cm} \tau\lc\sharp\mu_{j}\log\sharp\mu_{j}\rc{}.
\end{align}

\noindent Moreover, $1)$ in Theorem \ref{THM.Rel_Ent_AF_Cstar_Trace} and Equation \ref{EQ.COR.Rel_Ent_AF_Cstar_Trace_1} show $\Ent(\mu,\tau)\in (-\infty,\infty)$. It suffices to show the two equations above.\par
We show Equation \ref{EQ.COR.Rel_Ent_AF_Cstar_Trace_5} and Equation \ref{EQ.COR.Rel_Ent_AF_Cstar_Trace_6}. Compressing with projections decreases norm. Thus $3)$ in Proposition \ref{PRP.AF_Cstar_Trace_Dualisation_II} shows

\begin{align}\label{EQ.COR.Rel_Ent_AF_Cstar_Trace_7}
\sharp\mu=\bds\textrm{-}\lim_{j\in\mathbb{N}}\hspace{0.025cm} \sharp\mu_{j},\ \comp\sharp\mu=\bds\textrm{-}\lim_{j\in\mathbb{N}}\hspace{0.025cm} \comp\sharp\mu_{j}
\end{align}

\noindent by sequential strong continuity of multiplication, hence we additionally have uniform boundedness for all operators used in Equation \ref{EQ.COR.Rel_Ent_AF_Cstar_Trace_7}.\par


\pagebreak


Note Lemma \ref{LEM.FC_SR} requires such uniform boundedness. Using Lemma \ref{LEM.FC_SR}, we see Equation \ref{EQ.COR.Rel_Ent_AF_Cstar_Trace_7} shows

\begin{align}\label{EQ.COR.Rel_Ent_AF_Cstar_Trace_8}
\sharp\mu\log\sharp\mu=\s\textrm{-}\lim_{j\in\mathbb{N}}\hspace{0.025cm} \sharp\mu_{j}\log\sharp\mu_{j},\ \comp\sharp\mu\log\comp\sharp\mu=\bds\textrm{-}\lim_{j\in\mathbb{N}}\hspace{0.025cm} \comp\sharp\mu_{j}\log\comp\sharp\mu_{j}
\end{align}

\noindent since $\lambda\mapsto\lambda\log\lambda$ is continuous on $[0,\infty)$ \lc{}cf.~Remark \ref{REM.SR_Equivalence} and Remark \ref{REM.FC_SR}\rc{}. Using $\mu\in L^{1,\infty}(A,\tau)_{+}^{\flat}\cap\CI[p]$ and $\lim_{\lambda\rightarrow 0}\lambda\log\lambda=0$, $3)$ in Corollary \ref{COR.Wstar_Compression_Preservation_I} and Corollary \ref{COR.Wstar_Compression_Preservation_II} therefore imply

\begin{align}\label{EQ.COR.Rel_Ent_AF_Cstar_Trace_9}
\sharp\mu\log\sharp\mu=\comp\sharp\mu\log\comp\sharp\mu,\ \sharp\mu_{j}\log\sharp\mu_{j}=\comp\sharp\mu_{j}\log\comp\sharp\mu_{j}   
\end{align}

\noindent for all $j\in\mathbb{N}$. Equation \ref{EQ.COR.Rel_Ent_AF_Cstar_Trace_8} shows $\tau\lc\comp\sharp\mu\log\comp\sharp\mu\rc{}=\lim_{j\in\mathbb{N}}\tau\lc\comp\sharp\mu_{j}\log\comp\sharp\mu_{j}\rc$ by strong convergence since $\tau\lc{}p\rc{}<\infty$. Using the latter, Equation \ref{EQ.COR.Rel_Ent_AF_Cstar_Trace_8} and Equation \ref{EQ.COR.Rel_Ent_AF_Cstar_Trace_9} let us calculate

\begin{align*}
\tau\lc\sharp\mu\log\sharp\mu\rc{} & = \tau\lc\comp\sharp\mu\log\comp\sharp\mu\rc \phantom{\bigg)} \\
& = \lim_{j\in\mathbb{N}}\hspace{0.025cm} \tau\lc\comp\sharp\mu_{j}\log\comp\sharp\mu_{j}\rc \phantom{\bigg)} \\
& = \lim_{j\in\mathbb{N}}\hspace{0.025cm} \tau\lc\sharp\mu_{j}\log\sharp\mu_{j}\rc{}. \phantom{\bigg)}
\end{align*}

\noindent The above calculation shows Equation \ref{EQ.COR.Rel_Ent_AF_Cstar_Trace_5} and Equation \ref{EQ.COR.Rel_Ent_AF_Cstar_Trace_6} at once. Equation \ref{EQ.COR.Rel_Ent_AF_Cstar_Trace_1} and therefore $\Ent(\mu,\tau)\in (0,\infty)$ follows as discussed above.
\end{proof}


\subsubsection*{Finitely supported accessibility components}

Definition \ref{DFN.QOT_Distance_AC_FS} gives finitely supported fixed states and finitely supported accessibility components. The latter are defined by having finitely supported fixed state. Upon restriction, Theorem \ref{THM.QOT_Distance_AC_FS} shows we recover the strongly unital finite-trace case as per Theorem \ref{THM.Rel_Ent_AF_Cstar_Trace} depending on the given finitely supported fixed state. Theorem \ref{THM.L2W_EVI_Equivalence} uses Corollary \ref{COR.QOT_Distance_AC_FS}.\par
Let $(\phi,\bpsi,\gamma,\nabla)$ be noncommutative differential structure for tracial AF-$C^{*}$-algebras $(A,\tau)$ and $(B,\omega)$ in $\lc{}f,\theta\rc$-setting.

\begin{dfn}\label{DFN.QOT_Distance_AC_FS}
Let $\xi\in\SII(A)$ be a fixed state.

\begin{itemize}
\item[1)] We say that $\xi$ is finitely supported if $\xi\in\dom\Enttau$ and there exists a majorant of its local support.

\item[2)] We say that $\mathcal{C}\subset \big(\hspace{-0.03875cm} \SII(A),\mathcal{W}_{\nabla}^{\log}\big)$ is finitely supported with fixed part $\xi$ if $\CII$ has fixed part $\xi$ and the latter is finitely supported.
\end{itemize}
\end{dfn}


\pagebreak


\begin{thm}\label{THM.QOT_Distance_AC_FS}
Let $(\phi,\bpsi,\gamma,\nabla)$ be noncommutative differential structure for tracial AF-$C^{*}$-algebras $(A,\tau)$ and $(B,\omega)$ in $\lc{}f,\theta\rc$-setting. Let $\xi\in\SII(A)$ be a finitely supported fixed state. Let $\mathcal{C}\subset (\SII(A),\mathcal{W}_{\nabla}^{f,\theta})$ be finitely supported with fixed part $\xi$. Let $p\in L^{1,\infty}(A,\tau)$ be a projection s.t.~$\xi\in\CI[p]$.

\begin{itemize}
\item[1)] We have $\CII\subset\Fix_{A}(\xi)\subset\SII(A)\cap\CI[p]$ and $\CII\cap\dom\Enttau\subset\Fix_{A}^{\NI}(\xi)\subset\mathcal{S}^{\NI}(A)\cap\CI[p]$.

\item[2)] For a.e.~$j\in\mathbb{N}$, we have $\mathcal{C}_{A_{j}}\lc\bar{\xi}_{j}\rc\subset\Fix_{A_{j}}\lc\bar{\xi}_{j}\rc\subset\mathcal{S}^{\NI}(A)\cap\CI[p]$.

\item[3)] For all $\mu\in\Fix_{A}^{\NI}(\xi)$, we have

\begin{itemize}
\item[3.1)] $\supp\mu\leq p$ and $\mu$ has integrable support, \phantom{\big)}

\item[3.2)] $\Ent(\mu,\tau)=\lim_{j\in\mathbb{N}}\Ent\lc\mu_{j},\tau\rc{}=\lim_{j\in\mathbb{N}}\Ent\lc\bar{\mu}_{j},\tau\rc$. \phantom{\big)}
\end{itemize}

\begin{reapply}
\end{reapply}

\item[4)] $\Enttau:\Fix_{A}^{\NI}(\xi)\longrightarrow (-\infty,\infty]$ is l.s.c.~in $\mathcal{W}_{\nabla}^{f,\theta}$-topology.
\end{itemize}
\end{thm}
\begin{proof}
If $\supp\bar{\xi}_{j}=\supp\xi_{j}\leq p$ for $j\in\mathbb{N}$, then note $1)$ in Proposition \ref{PRP.AF_Cstar_Trace_Dualisation_III} for the tracial AF-$C^{*}$-algebra $(A[p],\tau)$ and $1)$ in Corollary \ref{COR.Support_Projection_I} for the tracial AF-$C^{*}$-algebra $(A_{j},\tau)$ yield inclusions

\begin{align}\label{EQ.THM.QOT_Distance_AC_FS_1}
\SII\big(A_{j,\bar{\xi}_{j}}\big)\subset\SII\lc{}A_{j}[p]\rc\subset\mathcal{S}^{\NI}(A[p]).
\end{align}

\noindent Using $1.3)$ in Theorem \ref{THM.Wstar_Derivation_QG_HSG_Regularity} and for a.e.~$j\in\mathbb{N}$, Equation \ref{EQ.THM.QOT_Distance_AC_FS_1} shows

\begin{align}\label{EQ.THM.QOT_Distance_AC_FS_2}
h_{t}\lc\mathcal{C}_{A_{j}}\lc\bar{\xi}_{j}\rc\rc\subset h_{t}\lc\textrm{Fix}_{A_{j}}\lc\bar{\xi}_{j}\rc\rc\subset\mathcal{S}_{-1}^{\NI,\infty}\big(A_{j,\bar{\xi}_{j}}\big)\subset\mathcal{S}^{\NI}(A[p])
\end{align}

\noindent for all $t>0$. Note $3)$ in Proposition \ref{PRP.Wstar_Derivation_QG_HSG_Fixed_Part_I} ensures the first inclusion in Equation \ref{EQ.THM.QOT_Distance_AC_FS_2}. Letting $t\downarrow 0$ in Equation \ref{EQ.THM.QOT_Distance_AC_FS_2} implies $2)$ by $1)$ in Proposition \ref{PRP.Wstar_Derivation_QG_HSG_II}. Using $2)$, we readily see $2)$ in Corollary \ref{COR.QOT_Distance_AC_I} yields

\begin{align}\label{EQ.THM.QOT_Distance_AC_FS_3}
\CII\subset\textrm{Fix}_{A}(\xi)\subset\SII(A)\cap\CI[p]
\end{align}

\noindent as per Diagram \ref{EQ.SSEC.QOT_QIT_Encoding_16} for $K=\dom\Enttau$. Using Lemma \ref{LEM.Rel_Ent_AF_Cstar_Trace_II}, Equation \ref{EQ.THM.QOT_Distance_AC_FS_3} shows

\begin{align}\label{EQ.THM.QOT_Distance_AC_FS_4}
\CII\cap\dom\Enttau\subset\textrm{Fix}_{A}^{\NI}(\xi)\subset\mathcal{S}^{\NI}(A)\cap\CI[p].
\end{align}

\noindent Equation \ref{EQ.THM.QOT_Distance_AC_FS_3} and Equation \ref{EQ.THM.QOT_Distance_AC_FS_4} show $1)$ at once. Using the latter, Lemma \ref{LEM.AF_Support_Projection_Majorant_Uniform} in turn implies $3.1)$, whereas $3)$ in Theorem \ref{THM.Rel_Ent_AF_Cstar_Trace} implies $3.2)$. Altogether, get $3)$. Using $1)$ in Theorem \ref{THM.QOT_Distance}, we readily see $2)$ in Theorem \ref{THM.Rel_Ent_AF_Cstar_Trace} shows $4)$.
\end{proof}

\begin{cor}\label{COR.QOT_Distance_AC_FS}
Let $\xi\in\SII(A)$ be a finitely supported fixed state. Let $\mathcal{C}\subset (\SII(A),\mathcal{W}_{\nabla}^{f,\theta})$ be finitely supported with fixed part $\xi$. We consider marginals $\mu^{0},\mu^{1}\in\mathcal{C}\cap\dom\Enttau$ and $(\mu,w)\in\Geo\lc\mu^{0},\mu^{1}\rc$ approximated by $\lc\mu^{j},w^{j}\rc_{j\in\mathbb{N}}\subset\Geo_{0}$ in finite dimensions. If there exists $C>0$ s.t.~for all $t\in (0,1)$, we have

\begin{align}\label{EQ.COR.QOT_Distance_AC_FS_1}
\Ent\lc\mu^{j}(t),\tau\rc\leq C\cdot \max\left\{\Ent\big(\bar{\mu}_{j}^{0},\tau\big),\Ent\big(\bar{\mu}_{j}^{1},\tau\big)\right\}
\end{align}

\noindent for a.e.~$j\in\mathbb{N}$, then $\mu(t)\in\dom\Enttau$ for all $t\in [0,1]$.
\end{cor}
\begin{proof}
Let $p\in L^{1,\infty}(A,\tau)$ be a projection s.t.~$\xi\in\CI[p]$. Let $m\in\mathbb{N}$ s.t.~$\lc\mu^{j},w^{j}\rc_{j\geq m}\subset\Geo_{0}$ as per Definition \ref{DFN.QOT_Minimiser_Approximation} for $(\mu,w)$. Using $2)$ in Theorem \ref{THM.QOT_Distance_AC_FS}, we assume $m\in\mathbb{N}$ s.t.~

\begin{align}\label{EQ.COR.QOT_Distance_AC_FS_2}
\mathcal{C}_{A_{j}}\lc\bar{\xi}_{j}\rc\subset\mathcal{S}^{\NI}(A)\cap\CI[p]
\end{align}

\noindent for all $j\geq m$ without loss of generality. Following $1)$ in Definition \ref{DFN.QOT_Minimiser_Approximation} and further $2)$ in Corollary \ref{COR.QOT_Distance_AC_I}, Equation \ref{EQ.COR.QOT_Distance_AC_FS_2} ensures we may in fact assume 

\begin{align}\label{EQ.COR.QOT_Distance_AC_FS_3}
\mu^{j}(t)\in\mathcal{C}_{A_{j}}\lc\bar{\xi}_{j}\rc\subset\mathcal{S}^{\NI}(A)\cap\CI[p]
\end{align}

\noindent for all $t\in [0,1]$ and $j\geq m$ without loss of generality. Note $\mu^{j}(0)=\mu_{j}^{0}$ and $\mu^{j}(1)=\mu_{j}^{1}$ in each case by hypothesis.\par
Following $2)$ in Definition \ref{DFN.QOT_Minimiser_Approximation}, we select a subsequence $\lc\mu^{j},w^{j}\rc_{j\geq m}$ converging to $(\mu,w)$ in $\Admnullone$. If there exists $C>0$ as per Equation \ref{EQ.COR.QOT_Distance_AC_FS_1}, then Equation \ref{EQ.COR.QOT_Distance_AC_FS_3} lets us apply $2)$ and $3)$ in Theorem \ref{THM.Rel_Ent_AF_Cstar_Trace} to Equation \ref{EQ.COR.QOT_Distance_AC_FS_1}. We calculate

\begin{align*}
\Ent\lc\mu(t),\tau\rc{} & \leq \liminf_{j\in\mathbb{N}}\hspace{0.025cm} \Ent\lc\mu^{j}(t),\tau\rc \phantom{\vstretch{1.15}{\Bigg)}} \\
& \leq C\cdot \max\bigg\{\liminf_{j\in\mathbb{N}}\hspace{0.025cm} \Ent\big(\bar{\mu}_{j}^{0},\tau\big),\liminf_{j\in\mathbb{N}}\hspace{0.025cm} \Ent\big(\bar{\mu}_{j}^{1},\tau\big)\bigg\} \phantom{\vstretch{1.15}{\Bigg)}} \\
& = C\cdot \max\bigg\{\lim_{j\in\mathbb{N}}\hspace{0.025cm} \Ent\big(\bar{\mu}_{j}^{0},\tau\big),\lim_{j\in\mathbb{N}}\hspace{0.025cm} \Ent\big(\bar{\mu}_{j}^{1},\tau\big)\bigg\} \phantom{\vstretch{1.15}{\Bigg)}} \\
& = C\cdot \max\bigg\{\Ent\lc\mu^{0},\tau\rc{},\Ent\lc\mu^{1},\tau\rc\bigg\}<\infty \phantom{\vstretch{1.15}{\Bigg)}}
\end{align*}

\noindent for all $t\in [0,1]$. Moreover, get $\Enttau>-\infty$ on $\SII(A)\cap\CI[p]$ by $1)$ in Theorem \ref{THM.Rel_Ent_AF_Cstar_Trace}.
\end{proof}


\pagebreak

 

\section{The logarithmic mean setting}\label{SEC.L2W_Log_Mean}

We use quantum relative entropy as measure of quantum information. Assume the logarithmic mean setting. Assuming finitely supported fixed parts, heat flow induces finite-energy admissible paths for all states with finite quantum relative entropy. Up to coarse graining, heat flow is gradient flow of quantum relative entropy. Heat flow further satisfies a steepest entropy ascent property \cite{ART.Ber_Con_Mon.2015.MaxEnt_SEA} by considering the steepest descent property of gradient flows in smooth Riemannian manifolds \cite{BK.Lan.1995.Riemannian_Manifolds} and taking limits. We seek conditions s.t.~steepest entropy ascent implies quantum noise evolution. If we are able to do so, then we obtain slopes of maximal entropy production, i.e.~erasure of quantum information, for sufficiently regular subsets of all bounded normal states. We accomplish this with our maximum entropy production principle \cite{ART.Dew.2003.MaxEnt_Information_I}\cite{ART.Dew.2005.MaxEnt_Information_II}\cite{ART.Mar_Sel.2006.MaxEnt_Review}.\par
In Subsection \ref{SSEC.L2W_EVI_Equivalence}, we consider heat flow as $\EVI_{\lambda}$-gradient flow of quantum relative entropy. We use Euler-Lagrange equations of energy functionals and results concerning Hessians of quantum relative entropy in the finite-dimensional setting. If heat flow is $\EVI_{\lambda}$-gradient flow of quantum relative entropy, then we have metric slopes as per Equation \ref{EQ.SSEC.L2W_EVI_Equivalence_1} \cite{BK.Amb_Gig_Sav.2008.Classical_OT_GradFlow}\cite{ART.Mur_Sav.2020.Classical_OT_EVI}. These generalise slopes of maximal entropy production,  even absent differential structures, to all normal states with finitely supported fixed part and finite quantum relative entropy. By locality, we restrict our maximum entropy production principle to selection of noise diffusion terms in the finite-dimensional setting and assume such selection is stable under scaling limits. We therefore view quantum Laplacians as generators of quantum noise evolution. In Subsection \ref{SSEC.L2W_EVI_Ric}, we use such description to show strictly positive lower Ricci bounds determine energy-information trade-offs parametrised by lower bounds on quantum noise.

\medskip

\noindent\textbf{Structure.} In Subsection \ref{SSEC.L2W_Log_Mean_QT}, we discuss fundamental properties of the logarithmic mean setting, define quantum $L^{2}$-Wasserstein distances and show heat flow induces finite-energy admissible paths. In Subsection \ref{SSEC.L2W_Log_Mean_Fin}, we show Euler-Lagrange equations and give, to us, necessary results concerning Hessians of quantum relative entropy. In Subsection \ref{SSEC.L2W_Log_Mean_QNE}, we formulate our maximum entropy production principle.


\subsection[Quantum $L^{2}$-Wasserstein distances]{Quantum $\mathbf{L}^{2}$-Wasserstein distances}\label{SSEC.L2W_Log_Mean_QT}

Quantum $L^{2}$-Wasserstein distances are quantum optimal transport distances in the logarithmic mean setting. Assuming the latter and finitely supported fixed parts, note Theorem \ref{THM.L2W_Log_Mean_NCDS} shows heat flow induces finite-energy admissible paths for all states with finite quantum relative entropy. Energy is controlled by time and relative entropy of marginals. Moreover, quantum relative entropy decreases along heat flow.


\subsubsection*{The logarithmic operator mean and representing function}

Definition \ref{DFN.L2W_Log_Mean_OM} gives the logarithmic operator mean. Equation \ref{EQ.DFN.L2W_Log_Mean_OM_1} induces the Kubo-Mori-Bogoliubov inner product \cite{ART.Pet_Tot.1993.Inner_Product}. Note Remark \ref{REM.L2W_Log_Mean_OM} for its functional derivative. Proposition \ref{PRP.L2W_Log_Mean_OM} gives its symmetric representing function. For details on such representing functions of operator means, we refer to Subsection \ref{SSEC.NCDS_NCD_QE}.

\begin{dfn}\label{DFN.L2W_Log_Mean_OM}
We define logarithmic operator mean $m_{\log}:(0,\infty)\times (0,\infty)\longrightarrow (0,\infty)$ by setting

\begin{align}\label{EQ.DFN.L2W_Log_Mean_OM_1}
m_{\log}(t,s):=\frac{t-s}{\log t-\log s}=\int_{0}^{1}t^{\alpha}s^{1-\alpha}d\alpha 
\end{align}

\noindent for all $t,s>0$.
\end{dfn}

\begin{rem}\label{REM.L2W_Log_Mean_OM}
Note $m_{\log}$ extends to $t,s\geq 0$ since $t\mapsto t^{\alpha}$ is monotone on $[0,\infty)$ for all $\alpha\in [0,1]$. We have $m_{\log}\lc{}0,0\rc{}=0$. Using functional derivative as per Definition \ref{DFN.QT_Derivation} and in the noncommutative chain rule given by Proposition \ref{PRP.Wstar_Derivation_Chain}, we have

\begin{align}\label{EQ.REM.L2W_Log_Mean_OM_1}
m_{\log}^{-1}(t,s)=\lc{}D\log\rc(t,s)=\int_{0}^{\infty}\lc{}t+\alpha 1\rc^{-1}\lc{}s+\alpha 1\rc^{-1}d\alpha
\end{align}

\noindent for all $t,s>0$. Integral characterisations of $m_{\log}$ and $m_{\log}^{-1}$ are well-known \cite{ART.Ped.2000.OpAlg_Diff_Functions}. 
\end{rem}

\begin{prp}\label{PRP.L2W_Log_Mean_OM}
We define $f_{\log}:(0,\infty)\longrightarrow (0,\infty)$ by setting

\begin{align}\label{EQ.PRP.L2W_Log_Mean_OM_1}
f_{\log}(t):=\frac{t-1}{\log t}    
\end{align}

\noindent for all $t>0$. Then $f_{\log}$ is the unique symmetric representing function s.t.~$m_{f_{\log}}=m_{\log}$.
\end{prp}
\begin{proof}
If $f_{\log}$ is a symmetric representing function s.t.~we have $m_{f_{\log}}(t,s)=f_{\log}(ts^{-1})s=m_{\log}(t,s)$ for all $t,s>0$, then Definition \ref{DFN.OM_Representing_Fct} yields our claim at once. We directly verify symmetry, as well as $f_{\log}(1)=1$ and $f_{\log}(ts^{-1})s=m_{\log}(t,s)$ for all $t,s>0$. The map $t\mapsto t^{\alpha}$ is operator monotone for all $\alpha\in [0,1]$. Since Equation \ref{EQ.PRP.L2W_Log_Mean_OM_1} is Equation \ref{EQ.DFN.L2W_Log_Mean_OM_1} for $s=1$ in each case, we know operator monotonicity of $f_{\log}$ by its integral characterisation.
\end{proof}


\subsubsection*{Definition and relation to quantum relative entropy}

Using symmetric representing function as per Proposition \ref{PRP.L2W_Log_Mean_OM}, Definition \ref{DFN.L2W_Log_Mean_NCDS} gives the logarithmic mean setting. Proposition \ref{PRP.L2W_Log_Mean_NCDS} shows the noncommutative chain rule intertwines logarithmic operator mean and noncommutative division operators. Equation \ref{EQ.PRP.L2W_Log_Mean_NCDS_1} links quantum optimal transport and noncommutative heat semigroups of quantum Laplacians. The latter uses both Notation \ref{NTN.Wstar_Derivation_QG_HSG_Regularity} and $1.1)$ in Corollary \ref{COR.Wstar_Derivation_QG_HSG_Regularity}. For details on compressing quantum gradients, we refer to Subsection \ref{SSEC.NCDS_NCG_Ubd_Derivation}.

\begin{dfn}\label{DFN.L2W_Log_Mean_NCDS}
Let $(\phi,\bpsi,\gamma,\nabla)$ be noncommutative differential structure for tracial AF-$C^{*}$-algebras $(A,\tau)$ and $(B,\omega)$ in $\lc{}f,\theta\rc$-setting. We are in the logarithmic mean setting if $f=f_{\log}$ represents $m_{\log}$ and $\theta=1$. We further say that it is finite-dimensional if $A$ and $B$ are finite-dimensional.
\end{dfn}

\begin{ntn}\label{NTN.L2W_Log_Mean_NCDS}
Assume the logarithmic mean setting. We write $\mathcal{I}^{\log}:=\mathcal{I}^{f,1}$, as well as $E^{\log}:=E^{f,1}$ and $\mathcal{W}_{\nabla}^{\log}:=\mathcal{W}_{\nabla}^{f,1}$.
\end{ntn}

\begin{prp}\label{PRP.L2W_Log_Mean_NCDS}
Assume the logarithmic mean setting. Let $\xi\in\mathcal{S}^{\NI}(A)$ be a fixed state with integrable support. If $x\in L^{\infty}(A_{\xi},\tau)_{\nabla}$ s.t.~$x>0$ in $L^{\infty}(A_{\xi},\tau)$, then $\log x\in L^{\infty}(A_{\xi},\tau)_{\nabla}$ and we have

\begin{align}\label{EQ.PRP.L2W_Log_Mean_NCDS_1}
\nabla_{\hspace{-0.055cm} \xi}\log x=\mathcal{D}_{x,\xi}\nabla_{\hspace{-0.055cm} \xi}x.
\end{align}
\end{prp}
\begin{proof}
Let $x\in L^{\infty}(A_{\xi},\tau)_{\nabla}$ s.t.~$x>0$ in $L^{\infty}(A_{\xi},\tau)$. The latter implies $0\notin\spec_{L^{\infty}(A_{\xi},\tau)} x$. Note $\spec_{L^{\infty}(A,\tau)}x=\spec_{L^{\infty}(A_{\xi},\tau)}x\cup\lset{}0\rset$ \lc{}cf.~$1)$ in Corollary \ref{COR.Wstar_Compression_Preservation_I}\rc{}. Let $I\subset [0,\infty)$ be a closed interval s.t.~$\spec_{L^{\infty}(A,\tau)}x\subset I$. Zero is isolated as both spectra are compact.\par
Let $g\in C^{1}(I)$ s.t.~$g(0)=0$ and $g(t)=\log t$ for all $t\in \spec_{L^{\infty}(A_{\xi},\tau)} x$. Such $g$ exists by the above discussion. We have $\log x=g(x)\in L^{\infty}(A_{\xi},\tau)$ \lc{}cf.~Corollary \ref{COR.Wstar_Compression_Preservation_II}\rc{}. Using $m_{\log}^{-1}(t,s)=\lc{}D\log\rc(t,s)$ for all $t,s>0$ as per Equation \ref{EQ.REM.L2W_Log_Mean_OM_1}, we calculate

\begin{align*}
\mathcal{D}_{x,\xi} & = m_{\log}^{-1}\lc{}L_{x,\supp\xi},R_{x,\supp\xi}\rc{} \phantom{\bigg)} \\
& = D\log\lc{}L_{x,\supp\xi},R_{x,\supp\xi}\rc{} \phantom{\bigg)} \\
& = Dg\lc{}L_{x,\supp\xi},R_{x,\supp\xi}\rc{}. \phantom{\bigg)}
\end{align*}

\noindent Following $1.1)$ in Corollary \ref{COR.Wstar_Derivation_QG_HSG_Regularity}, compressing $\nabla:A_{0}\longrightarrow L^{2}(B,\omega)$ with $\supp\xi$ yields symmetric $W^{*}$-derivation $\nabla_{\hspace{-0.055cm} \xi}:\AII_{\xi}\longrightarrow L^{2}(B_{\xi},\omega)$. Using the above calculation in order to account for $\mathcal{D}_{x,\xi}$, applying $1)$ in Proposition \ref{PRP.Wstar_Derivation_Chain} for $\nabla_{\hspace{-0.055cm} \xi}$ to $g$ selected as above shows $\log x=g(x)\in L^{\infty}(A_{\xi},\tau)_{\nabla}$ and furthermore Equation \ref{EQ.PRP.L2W_Log_Mean_NCDS_1}.
\end{proof}

\begin{rem}\label{REM.L2W_Log_Mean_NCDS_I}
Note $\nabla=\nabla_{\hspace{-0.055cm} \xi}$ on $L^{\infty}(A_{\xi},\tau)_{\nabla}$ by $4.2)$ in Corollary \ref{COR.Wstar_Derivation_QG_HSG_Regularity}. We may therefore suppress $\xi$ in the subscript of $\nabla_{\hspace{-0.055cm} \xi}$ in Equation \ref{EQ.PRP.L2W_Log_Mean_NCDS_1}.
\end{rem}

Theorem \ref{THM.L2W_Log_Mean_NCDS} uses Lemma \ref{LEM.L2W_Log_Mean_NCDS}, resp.~its immediate Corollary \ref{COR.L2W_Log_Mean_NCDS}. Up to coarse graining, Lemma \ref{LEM.L2W_Log_Mean_NCDS} implies heat flow is gradient flow of quantum relative entropy. Moreover, Theorem \ref{THM.L2W_Log_Mean_QNE} generalises arguments in their proof without assuming the finite-dimensional setting. We use operator differentiable functions \cite{ART.Ped.2000.OpAlg_Diff_Functions}. We review its general case. Let $H$ be a separable Hilbert space and $T>0$ in $\BII(H)$. Equation \ref{EQ.SSEC.L2W_Log_Mean_QT_1} uses Fr\'echet derivatives in $\BII(H)$. For all $S\in\BII(H)$, we define $d\log_{\hspace{0.05cm}T}(S)\in\BII(H)$ by setting

\begin{align}\label{EQ.SSEC.L2W_Log_Mean_QT_1}
d\log_{\hspace{0.05cm}T}(S):=\restr{0.925}{\frac{d}{dt}}{t=0,\|.\|_{\BII(H)}}\log\varphi(t)
\end{align}

\noindent for all Fr\'echet differentiable maps $t\mapsto \varphi(t)\in\BII(H)_{>0}$ s.t.~$\varphi(0)=T$ and $\dot{\varphi}(0)=S$. We obtain $d\log_{T}:\BII(H)\longrightarrow\BII(H)$. For all $S\in\BII(H)$, we have

\begin{align}\label{EQ.SSEC.L2W_Log_Mean_QT_2}
d\log_{\hspace{0.05cm}T}(S)=\int_{0}^{\infty}\lc\alpha I+T\rc^{-1}S\lc\alpha I+T\rc^{-1}d\alpha,\ \int_{0}^{1}T^{\alpha}d\log_{\hspace{0.05cm}T}(S)T^{1-\alpha}d\alpha=S
\end{align}

\noindent by Subsection 4.3 in \cite{ART.Ped.2000.OpAlg_Diff_Functions}. Identities in Equation \ref{EQ.SSEC.L2W_Log_Mean_QT_2} determine $d\log_{T}$. Equation \ref{EQ.SSEC.L2W_Log_Mean_QT_1} and Equation \ref{EQ.SSEC.L2W_Log_Mean_QT_2} pull back along compressed canonical left-~and right-action as given below in the proof of Lemma \ref{LEM.L2W_Log_Mean_NCDS}.

\begin{lem}\label{LEM.L2W_Log_Mean_NCDS}
Assume the finite-dimensional logarithmic mean setting. Let $\xi\in\SII(A)$ be a fixed state and $[a,b]\subset\mathbb{R}$.

\begin{itemize}
\item[1)] If $\mu:[a,b]\longrightarrow\vartheta(\xi)$ is differentiable for a.e.~$t\in [a,b]$, then 

\begin{align}\label{EQ.LEM.L2W_Log_Mean_NCDS_1}
\frac{d}{dt}\Enttau\lc\mu(t)\rc{}=\tau\lc\sharp\dot{\mu}(t)\log\sharp\mu(t)\rc{}
\end{align}

\begin{reapply}
\end{reapply}

\noindent for a.e.~$t\in [a,b]$.

\item[2)] If $(\mu,w)\in\Admab$ s.t.~$\mu(t)\in\vartheta(\xi)$ for all $t\in [a,b]$ and furthermore $\sharp w(t)\in B_{\xi}$ for a.e.~$t\in [a,b]$, then 

\begin{align}\label{EQ.LEM.L2W_Log_Mean_NCDS_2}
\frac{d}{dt}\Enttau\lc\mu(t)\rc{}=\tau\lc\sharp\dot{\mu}(t)\log\sharp\mu(t)\rc{}=\lgl\mathcal{D}_{\sharp\mu(t),\xi}\sharp w(t),\nabla\sharp\mu(t)\rgl_{\omega}
\end{align}

\begin{reapply}
\end{reapply}

\noindent for a.e.~$t\in [a,b]$.
\end{itemize}
\end{lem}
\begin{proof}
We use Hilbert space $\lc{}A_{\xi},\|.\|_{\tau}\rc$. Pull-back along compressed canonical left-~and right-action preserves Fr\'echet derivatives and therefore identities as above. These use $A_{\xi,>0}$ and $A_{\xi}$, rather than $\BII(A_{\xi})_{>0}$ and $\BII(A_{\xi})$. If $\mu:[a,b]\longrightarrow\vartheta(\xi)$ is differentiable for a.e.~$t\in [a,b]$, then Proposition \ref{PRP.RM_Embedded_Submanifold_II} ensures $\sharp\mu(t)>0$ in $A_{\xi}$ for all $t\in [a,b]$ and $\sharp\dot{\mu}(t)\in I(\Delta_{\xi})$ for a.e.~$t\in [a,b]$. Thus $\log\sharp\mu(t)\in A_{\xi}$ in each case \lc{}cf.~Corollary \ref{COR.Wstar_Compression_Preservation_II}\rc{}, hence the map $s\mapsto\log\sharp\mu(s)\in A_{\xi}$ is Fr\'echet differentiable for all $t\in (a,b)$. Corollary \ref{COR.Rel_Ent_AF_Cstar_Trace} shows $\Ent\lc\mu(t),\tau\rc{}=\tau\lc\sharp\mu(t)\log\sharp\mu(t)\rc$ for all $t\in [a,b]$. Note $2)$ in Proposition \ref{PRP.RM_Embedded_Submanifold_I} implies $I(\Delta_{\xi})\subset\ker\tau$. Using the latter, we argue as follows.\par
We show $1)$. Assume its setting. If $\dot{\mu}(t)$ exists for $t\in [a,b]$, then traciality, the second identity in Equation \ref{EQ.SSEC.L2W_Log_Mean_QT_2},  and $\sharp\dot{\mu}(t)\in\ker\tau$ imply

\begin{align}\label{EQ.LEM.L2W_Log_Mean_NCDS_3}
\tau\lc\sharp\mu(t)d\log_{\sharp\mu(t)}\lc\sharp\dot{\mu}(t)\rc\rc{}=\int_{0}^{1}\tau\lc\sharp\mu(t)^{\alpha}d\log_{\sharp\mu(t)}\lc\sharp\dot{\mu}(t)\rc\sharp\mu(t)^{1-\alpha}\rc{}d\alpha=\tau\lc\sharp\dot{\mu}(t)\rc{}=0.
\end{align}

\noindent We know $\Ent\lc\mu(t),\tau\rc{}=\tau\lc\sharp\mu(t)\log\sharp\mu(t)\rc$ for all $t\in [a,b]$. Using the latter and the Leibniz rule for Fr\'echet derivatives, Equation \ref{EQ.LEM.L2W_Log_Mean_NCDS_3} lets us calculate

\begin{align}\label{EQ.LEM.L2W_Log_Mean_NCDS_4}
\frac{d}{dt}\Enttau\lc\mu(t)\rc{}=\tau\lc\sharp\dot{\mu}(t)\log\sharp\mu(t)\rc{}+\tau\lc\sharp\mu(t)d\log_{\sharp\mu(t)}\lc\sharp\dot{\mu}(t)\rc\rc{}=\tau\lc\sharp\dot{\mu}(t)\log\sharp\mu(t)\rc{}
\end{align}

\noindent for a.e.~$t\in [a,b]$. Equation \ref{EQ.LEM.L2W_Log_Mean_NCDS_4} shows Equation \ref{EQ.LEM.L2W_Log_Mean_NCDS_1}. We obtain $1)$. In particular, we see $1)$ applies to all elements in $\Admab$.\par


\pagebreak


We show $2)$. Assume its setting. Note finite-dimensionality ensures $\xi$ has integrable support. Using Proposition \ref{PRP.L2W_Log_Mean_NCDS}, Equation \ref{EQ.LEM.L2W_Log_Mean_NCDS_1} lets us calculate

\begin{align}\label{EQ.LEM.L2W_Log_Mean_NCDS_5}
\frac{d}{dt}\Enttau\lc\mu(t)\rc{}=\tau\lc\sharp\dot{\mu}(t)\log\sharp\mu(t)\rc{}=\lgl\sharp w(t),\mathcal{D}_{\sharp\mu(t),\xi}\nabla\sharp\mu(t)\rgl_{\omega}
\end{align}

\noindent for a.e.~$t\in [a,b]$. We suppress $\xi$ in the subscript in Equation \ref{EQ.LEM.L2W_Log_Mean_NCDS_5} as per Remark \ref{REM.L2W_Log_Mean_NCDS_I}. We swap $\mathcal{D}_{\sharp\mu(t),\xi}\in\BII(B_{\xi})_{h}$ to the left-hand side of the inner product. Note this requires $\sharp w(t)\in B_{\xi}$. Equation \ref{EQ.LEM.L2W_Log_Mean_NCDS_5} shows Equation \ref{EQ.LEM.L2W_Log_Mean_NCDS_2}. We have $2)$.
\end{proof}

\begin{cor}\label{COR.L2W_Log_Mean_NCDS}
Assume the finite-dimensional logarithmic mean setting. Let $\xi\in\SII(A)$ be a fixed state. For all $\mu\in\Fix_{A}(\xi)$ and $t>0$, we have

\begin{itemize}
\item[1)] $-\frac{d}{dt}\Enttau\lc{}h_{t}(\mu)\rc{}=\tau\lc\Delta h_{t}\lc\sharp\mu\rc\log h_{t}\lc\sharp\mu\rc\rc$, \phantom{\vstretch{0.75}{\bigg)}}

\item[2)] $\tau\lc\Delta h_{t}\lc\sharp\mu\rc\log h_{t}\lc\sharp\mu\rc\rc{}=\big\|\mathcal{D}_{h_{t}\lc\sharp\mu\rc{},\xi}^{\frac{1}{2}}\nabla h_{t}\lc\sharp\mu\rc\big\|_{\omega}^{2}=\mathcal{I}^{\log}\lc{}h_{t}(\mu),h_{t}(\mu),\lc\nabla h_{t}\lc\sharp\mu\rc\rc^{\flat}\rc$. \phantom{\vstretch{0.75}{\bigg)}}
\end{itemize}
\end{cor}
\begin{proof}
For all $t\in\mathbb{R}$, set $\mu(t):=h_{t}(\mu)$ and $w(t):=-\lc\nabla\sharp\mu(t)\rc^{\flat}$. Thus $\sharp\dot{\mu}(t)=-\Delta\sharp\mu(t)=\nabla^{*}\sharp w(t)$ for all $t>0$ by definition of $h:[0,\infty)\longrightarrow\BII(A)$ and construction of extension $h:[0,\infty)\longrightarrow\BII(A^{*})$, hence $(\mu,w)\in\Adm^{[a,b]}$ for all $[a,b]\subset\mathbb{R}$ by finite-dimensionality.\par
If $t>0$, then $\sharp\mu(t)>0$ in $A_{\xi}$ by $2.2)$ in Theorem \ref{THM.Wstar_Derivation_QG_HSG_Regularity} and furthermore $\sharp w(t)\in B_{\xi}$ by $1.1)$ in Corollary \ref{COR.Wstar_Derivation_QG_HSG_Regularity}. Using the latter and $1.2)$ in Corollary \ref{COR.Wstar_Derivation_QG_HSG_Regularity}, applying $2)$ in Lemma \ref{LEM.L2W_Log_Mean_NCDS} for all $[a,b]\subset [0,\infty)$ yields both $1)$ and the first identity in $2)$ at once. Note $1)$ in Proposition \ref{PRP.RM_Compressed_PMO} shows the second one immediately.
\end{proof}

\begin{thm}\label{THM.L2W_Log_Mean_NCDS}
Let $(\phi,\bpsi,\gamma,\nabla)$ be noncommutative differential structure for tracial AF-$C^{*}$-algebras $(A,\tau)$ and $(B,\omega)$ in the logarithmic mean setting. Let $\xi\in\SII(A)$ be a finitely supported fixed state. For all $\mu\in\Fix_{A}^{\NI}(\xi)\cap\dom\Enttau$ and $t\geq 0$, we have

\begin{itemize}
\item[1)] $h_{t}(\mu)\in\dom\Enttau$, \phantom{\vstretch{0.875}{\bigg)}}

\item[2)] $\Ent(\xi,\tau)\leq\Ent\lc{}h_{t}(\mu),\tau\rc\leq\Ent(\mu,\tau)$, \phantom{\vstretch{0.875}{\bigg)}}

\item[3)] $\mathcal{W}_{\nabla}^{\log}\lc\mu,h_{t}(\mu)\rc^{2}\leq t\cdot \bigg(\hspace{-0.028975cm} \Ent(\mu,\tau)-\Ent\lc{}h_{t}(\mu),\tau\rc\bigg)<\infty$. \phantom{\vstretch{0.875}{\bigg)}}
\end{itemize}
\end{thm}
\begin{proof}
Note $2)$ implies $1)$. We show $2)$ and $3)$. Let $\mu\in\Fix_{A}^{\NI}(\xi)\cap\dom\Enttau$. If $\xi_{j}\neq 0$ for $j\in\mathbb{N}$, then $\overline{h_{t}(\mu)}_{j}=h_{t}\lc\bar{\mu}_{j}\rc\in\SII\big(A_{j,\bar{\xi}_{j}}\big)$ for all $t\in [0,\infty]$ by $1.3)$ in Proposition \ref{PRP.Wstar_Derivation_QG_HSG_Fixed_Part_I} and $1.3)$ in Theorem \ref{THM.Wstar_Derivation_QG_HSG_Regularity}. Using the latter and $1.2)$ in Proposition \ref{PRP.AF_Cstar_Trace_Dualisation_II}, we reduce to the finite-dimensional setting. For all $t\in [0,\infty]$, we have

\begin{align}\label{EQ.THM.L2W_Log_Mean_NCDS_1}
\Ent\lc{}h_{t}(\mu),\tau\rc{}=\lim_{j\in\mathbb{N}}\hspace{0.025cm} \Ent\lc{}h_{t}\lc\bar{\mu}_{j}\rc{},\tau\rc{}
\end{align}

\noindent by $3)$ in Theorem \ref{THM.QOT_Distance_AC_FS}, as well as

\begin{align}\label{EQ.THM.L2W_Log_Mean_NCDS_2}
\mathcal{W}_{\nabla}^{\log}\lc\mu,h_{t}(\mu)\rc{}=\lim_{j\in\mathbb{N}}\hspace{0.025cm} \mathcal{W}_{\nabla_{\hspace{-0.055cm} j}}^{\log}\lc\bar{\mu}_{j},h_{t}\lc\bar{\mu}_{j}\rc\rc{} 
\end{align}

\noindent by $1)$ and $3)$ in Theorem \ref{THM.QOT_Distance}. We use both equations above to reduce as follows.\par
Equation \ref{EQ.THM.L2W_Log_Mean_NCDS_1} by itself shows $2)$ if for a.e.~$j\in\mathbb{N}$, we have

\begin{align}\label{EQ.THM.L2W_Log_Mean_NCDS_3}
\Ent\lc\bar{\xi}_{j},\tau\rc\leq\Ent\lc{}h_{t}\lc\bar{\mu}_{j}\rc{},\tau\rc\leq\Ent\lc\bar{\mu}_{j},\tau\rc{}
\end{align}

\noindent for all $t\geq 0$. Equation \ref{EQ.THM.L2W_Log_Mean_NCDS_1} and Equation \ref{EQ.THM.L2W_Log_Mean_NCDS_2} show $3)$ if for a.e.~$j\in\mathbb{N}$, we have

\begin{align}\label{EQ.THM.L2W_Log_Mean_NCDS_4}
\mathcal{W}_{\nabla_{\hspace{-0.055cm} j}}^{\log}\lc\bar{\mu}_{j},h_{t}\lc\bar{\mu}_{j}\rc\rc^{2}\leq t\cdot \bigg(\hspace{-0.028975cm} \Ent\lc\bar{\mu}_{j},\tau\rc{}-\Ent\lc{}h_{t}\lc\bar{\mu}_{j}\rc{},\tau\rc\bigg)
\end{align}

\noindent for all $t\geq 0$. Since $\xi_{j}\neq 0$ for a.e.~$j\in\mathbb{N}$, Equation \ref{EQ.THM.L2W_Log_Mean_NCDS_3} reduces $2)$ and Equation \ref{EQ.THM.L2W_Log_Mean_NCDS_4} reduces $3)$ to the finite-dimensional setting.\par
Assume $A$ and $B$ are finite-dimensional. In particular, $\xi\neq 0$. We show $2)$. Note $2.3)$ in Proposition \ref{PRP.Wstar_Derivation_QG_HSG_II} ensures $\sup_{t\in [0,\infty]}\dblv{}h_{t}(\mu)\dblv_{\infty}\leq \|\mu\|_{\infty}$. Get compact $K\subset [0,\infty)$ s.t.~$\specA h_{t}(\mu)\subset K$ for all $t\in [0,\infty]$. Let $g\in C_{b}(\mathbb{R})$ s.t.~$g(\lambda)=\lambda\log\lambda$ for all $\lambda\in K$. Using Lemma \ref{LEM.FC_SR} and Corollary \ref{COR.Rel_Ent_AF_Cstar_Trace}, we see $g\in C_{b}(\mathbb{R})$ ensures we define continuous map $F:[0,\infty)\longrightarrow\mathbb{R}$ by setting

\begin{align}\label{EQ.THM.L2W_Log_Mean_NCDS_5}
t\mapsto F(t):=\Ent\lc{}h_{t}(\mu),\tau\rc{}=\tau\lc\sharp h_{t}(\mu)\log\sharp h_{t}(\mu)\rc{}=\tau\lc{}g\lc\sharp h_{t}(\mu)\rc\rc{}
\end{align}

\noindent for all $t\geq 0$ \lc{}cf.~Remark \ref{REM.SR_Equivalence} and Remark \ref{REM.FC_SR}\rc{}. This requires strong continuity as per $1)$ in Proposition \ref{PRP.Wstar_Derivation_QG_HSG_II}. We obtain $\Ent(\xi,\tau)=\lim_{t\rightarrow\infty}F(t)$. Corollary \ref{COR.L2W_Log_Mean_NCDS} shows

\begin{align}\label{EQ.THM.L2W_Log_Mean_NCDS_6}
-\frac{d}{dt}F(t)=\dblv{}\mathcal{D}_{h_{t}\lc\sharp\mu\rc{},\xi}^{\frac{1}{2}}\nabla h_{t}\lc\sharp\mu\rc\dblv_{\omega}^{2}=\mathcal{I}^{\log}\lc{}h_{t}(\mu),h_{t}(\mu),\lc\nabla h_{t}\lc\sharp\mu\rc\rc^{\flat}\rc\geq 0
\end{align}

\noindent for all $t>0$. Equation \ref{EQ.THM.L2W_Log_Mean_NCDS_5} and Equation \ref{EQ.THM.L2W_Log_Mean_NCDS_6} show $2)$.\par
We show $3)$. If $t=0$, then our claim follows since $\mathcal{W}_{\nabla}^{\log}$ is a metric. Assume $t>0$. For all $s\in [0,t]$, set $\mu(s):=h_{s}(\mu)$ and $w(s):=-\lc\nabla\sharp\mu(s)\rc^{\flat}$. We show $(\mu,w)\in\Adm^{[0,t]}\lc\mu,h_{t}(\mu)\rc$ in the proof of Corollary \ref{COR.L2W_Log_Mean_NCDS}. Using the map $s\mapsto\varphi(s):=ts$, we rescale $(\mu,w)\in\Adm^{[0,t]}$ to $\lc\mu',w'\rc\in\Admnullone$ as per Remark \ref{REM.Energy_Functional_Reparametrisation}. Using the map $s\mapsto\varphi^{-1}(s):=t^{-1}s$, we likewise rescale $\lc\mu',w'\rc\in\Admnullone$ to $(\mu,w)\in\Adm^{[0,t]}$. We further apply Proposition \ref{PRP.Energy_Functional_Reparametrisation} to the latter below. Equation \ref{EQ.THM.L2W_Log_Mean_NCDS_6} lets us calculate

\begin{align*}
\mathcal{W}_{\nabla}^{\log}\lc\mu,h_{t}(\mu)\rc^{2} & \leq E^{\log}\lc\mu',w'\rc{} \phantom{\bigg)} \\
& = t\cdot \int_{0}^{t}\mathcal{I}^{\log}\lc\mu(s),w(s),w(s)\rc{}ds \phantom{\bigg)} \\
& = t\cdot \bigg(\hspace{-0.028975cm} \Ent(\mu,\tau)-\Ent\lc{}h_{t}(\mu),\tau\rc\bigg). \phantom{\bigg)}
\end{align*}

\noindent Note $2)$ ensures the right-hand side is finite. As such, the above calculation shows $3)$ at once. The general case follows as discussed above.
\end{proof}


\subsection{The finite-dimensional setting}\label{SSEC.L2W_Log_Mean_Fin}

We discuss the finite-dimensional logarithmic mean setting. We introduce $\Lambda$-operations to simplify those of our calculations involving derivatives and noncommutative division operators. Theorem \ref{THM.L2W_Log_Mean_NCDS_Fin_EL} formulates Euler-Lagrange equations. Theorem \ref{THM.L2W_Log_Mean_NCDS_Fin_Hessian} gives two differential equations for Hessians of quantum relative entropy.


\subsubsection*{Euler-Lagrange equations}

Theorem \ref{THM.L2W_Log_Mean_NCDS_Fin_Hessian} requires Euler-Lagrange equations as per Theorem \ref{THM.L2W_Log_Mean_NCDS_Fin_EL}. We introduce $\Lambda$-operations in order to simplify calculations. Let $(\phi,\bpsi,\gamma,\nabla)$ be noncommutative differential structure for tracial AF-$C^{*}$-algebras $(A,\tau)$ and $(B,\omega)$ in the finite-dimensional logarithmic mean setting. Let $\xi\in\SII(A)$ be a fixed state. Finite-dimensionality ensures finite support.\par
Following Definition \ref{DFN.L2W_Log_Mean_OM} and Remark \ref{REM.L2W_Log_Mean_OM}, note Proposition \ref{PRP.L2W_Log_Mean_NCDS_Fin_EL_I} gives integral characterisation of multiplication and division operators in our setting. These allow for direct calculations. We moreover obtain smooth maps in Definition \ref{DFN.L2W_Log_Mean_NCDS_Fin_EL_I}.

\begin{prp}\label{PRP.L2W_Log_Mean_NCDS_Fin_EL_I}
For all $x\in A_{\xi,+}$ and $u\in B_{\xi}$, we have

\begin{itemize}
\item[1)] $\mathcal{M}_{x,\xi}(u)=\int_{0}^{1}\phi(x)^{\alpha}u\bpsi(x)^{1-\alpha}d\alpha$,

\item[2)] $\mathcal{D}_{x,\xi}(u)=\int_{0}^{\infty}\lc\phi(x)+\alpha 1_{B}\rc^{-1}u\lc\bpsi(x)+\alpha 1_{B}\rc^{-1}d\alpha$ if $x>0$ in $A_{\xi}$.
\end{itemize}
\end{prp}
\begin{proof}
We show $1)$. Let $x\in A_{\xi,+}$ and $u\in B_{\xi}$. Using $1)$ in Lemma \ref{LEM.NCD_Operator_Compressed_PMO_I}, Equation \ref{EQ.DFN.L2W_Log_Mean_OM_1} lets us calculate

\begin{align*}
\mathcal{M}_{x,\xi}(u)=\mathcal{M}_{x}(u)\phantom{\bigg)} & = m_{\log}\lc{}L_{x}^{\phi},R_{x}^{\bpsi}\rc{}(u) \phantom{\vstretch{1.075}{\Bigg)}} \\
& = \int_{0}^{1}\big(L_{x}^{\phi}\big)^{\alpha}\lc\big(R_{x}^{\bpsi}\big)^{1-\alpha}(u)\rc{}d\alpha \phantom{\vstretch{1.075}{\Bigg)}} \\
& = \int_{0}^{1}\phi(x)^{\alpha}u\bpsi(x)^{1-\alpha}d\alpha. \phantom{\vstretch{1.075}{\Bigg)}}
\end{align*}

The above calculation shows $1)$. We show $2)$. Assume $x>0$ in $A_{\xi}$. For all $\alpha>0$, we define $g^{\alpha}\in C_{b}\lc{}[0,\infty)\times [0,\infty)\rc$ by setting 

\begin{align}\label{EQ.PRP.L2W_Log_Mean_NCDS_Fin_EL_I_1}
g^{\alpha}(t,s):=\big(t+\alpha\big)^{-1}\big(s+\alpha\big)^{-1}
\end{align}

\noindent for all $t,s\geq 0$. Since $u\in B_{\xi}$, $2.3)$ in Lemma \ref{LEM.AF_NCD_FC_I} implies

\begin{align}\label{EQ.PRP.L2W_Log_Mean_NCDS_Fin_EL_I_2}
g^{\alpha}\lc{}L_{x}^{\phi},R_{x}^{\bpsi}\rc{}(u)=g^{\alpha}\lc{}L_{x,\supp\xi}^{\phi},R_{x,\supp\xi}^{\bpsi}\rc{}(u)
\end{align}

\noindent for all $\alpha>0$. Using $2)$ in Lemma \ref{LEM.NCD_Operator_Compressed_PMO_I}, Equation \ref{EQ.REM.L2W_Log_Mean_OM_1} and Equation \ref{EQ.PRP.L2W_Log_Mean_NCDS_Fin_EL_I_2} show

\begin{align*}
\mathcal{D}_{x,\xi}(u)=\lc\mathcal{M}_{x,\xi}\rc^{-1}(u) & = m_{\log}^{-1}\lc{}L_{x,\supp\xi}^{\phi},R_{x,\supp\xi}^{\bpsi}\rc{}(u) \phantom{\Bigg)} \\
& = \int_{0}^{\infty}\lc{}L_{x,\supp\xi}^{\phi}+\alpha I_{B_{\xi}}\rc^{-1}\lc\lc{}R_{x,\supp\xi}^{\bpsi}+\alpha I_{B_{\xi}}\rc^{-1}(u)\rc{}d\alpha \phantom{\Bigg)} \\
& = \int_{0}^{\infty}g^{\alpha}\lc{}L_{x,\supp\xi}^{\phi},R_{x,\supp\xi}^{\bpsi}\rc{}(u)d\alpha \phantom{\Bigg)} \\
& = \int_{0}^{\infty}g^{\alpha}\lc{}L_{x}^{\phi},R_{x}^{\bpsi}\rc{}(u)d\alpha \phantom{\Bigg)} \\
& = \int_{0}^{\infty}\lc{}L_{x}^{\phi}+\alpha I_{B}\rc^{-1}\lc\lc{}R_{x}^{\bpsi}+\alpha I_{B}\rc^{-1}(u)\rc{}d\alpha \phantom{\Bigg)} \\
& = \int_{0}^{\infty}\lc\phi(x)+\alpha 1_{B}\rc^{-1}u\lc\bpsi(x)+\alpha 1_{B}\rc^{-1}d\alpha. \phantom{\Bigg)}
\end{align*}

\noindent Get $2)$. Note $x>0$ in $A_{\xi}$ is required for $\mathcal{D}_{x,\xi}=\lc\mathcal{M}_{x,\xi}\rc^{-1}$ to be defined.
\end{proof}

We have Riemannian manifold $(\vartheta(\xi),g^{\xi})$ as per Definition \ref{DFN.RM_II} embedded in $\SII(A_{\xi})$ as per Proposition \ref{PRP.RM_Embedded_Submanifold_II} s.t.~its tangent bundle is indeed trivial with fibre $I(\Delta_{\xi})$. For all $\mu\in\vartheta(\xi)$, we introduce operators 

\begin{align}\label{EQ.SSEC.L2W_Log_Mean_Fin_1}
\mathfrak{F}_{\mu}=\nabla^{*}\mathcal{M}_{\sharp\mu,\xi}\nabla=\nabla^{*}\mathcal{M}_{\sharp\mu}\nabla,\ \mathfrak{G}_{\mu}=\mathcal{M}_{\sharp\mu,\xi}\nabla=\mathcal{M}_{\sharp\mu}\nabla
\end{align}

\noindent with domain $\im\Delta_{\xi}$ as per Definition \ref{DFN.RM_I} by Proposition \ref{PRP.RM_I}. In each case, get $\mathfrak{F}_{\mu}>0$ in $\BII(\im\Delta_{\xi})$ s.t.~$[\mathfrak{F}_{\mu},\Adj]=0$. Moreover, note $\mathfrak{G}_{\mu}\in\BII(\im\Delta_{\xi},B_{\xi})$ intertwines adjoining. We therefore restrict to $I(\Delta_{\xi})=\im\Delta_{\xi}\cap A_{\xi,h}$.

\begin{ntn}
Let $X$ and $Y$ be smooth manifolds. We write $dg:TX\longrightarrow TY$ for the first differential form of a smooth map $g:X\longrightarrow Y$ \cite{BK.Lan.1995.Riemannian_Manifolds}, i.e.~its total derivative. We further write $d_{\mu}g\in\textrm{Hom}\lc{}T_{\mu}X,T_{g(\mu)}Y\rc$ upon evaluation at $\mu\in X$.
\end{ntn}

\begin{dfn}\label{DFN.L2W_Log_Mean_NCDS_Fin_EL_I}
We consider Riemannian manifold $(\vartheta(\xi),g^{\xi})$.

\begin{itemize}
\item[1)] We define $\mathcal{M}_{\xi}:\vartheta(\xi)\longrightarrow\GL(\BII(B_{\xi}))$ by setting $\mathcal{M}_{\xi}(\mu):=\mathcal{M}_{\sharp\mu,\xi}$ for all $\mu\in\vartheta(\xi)$. We define $\mathcal{D}_{\xi}:\vartheta(\xi)\longrightarrow\GL(\BII(B_{\xi}))$ by setting $\mathcal{D}_{\xi}(\mu):=\mathcal{D}_{\sharp\mu,\xi}$ for all $\mu\in\vartheta(\xi)$.

\item[2)] We define $\mathfrak{F}:\vartheta(\xi)\longrightarrow\GL(\BII(\im\Delta_{\xi}))$ by setting $\mathfrak{F}(\mu):=\mathfrak{F}_{\mu}$ for all $\mu\in\vartheta(\xi)$. We define $\mathfrak{G}:\vartheta(\xi)\longrightarrow\BII(\im\Delta_{\xi},B_{\xi})$ by setting $\mathfrak{G}(\mu):=\mathfrak{G}_{\mu}$ for all $\mu\in\vartheta(\xi)$.
\end{itemize}
\end{dfn}

\begin{rem}
Proposition \ref{PRP.L2W_Log_Mean_NCDS_Fin_EL_I} shows all maps in Definition \ref{DFN.L2W_Log_Mean_NCDS_Fin_EL_I} are smooth. We use this throughout our discussion.
\end{rem}

Definition \ref{DFN.L2W_Log_Mean_NCDS_Fin_EL_II} gives $\Lambda$-operations. Proposition \ref{PRP.L2W_Log_Mean_NCDS_Fin_EL_II} and Lemma \ref{LEM.L2W_Log_Mean_NCDS_Fin_EL} simplify calculations involving derivatives and noncommutative division operators. Using their results, Theorem \ref{THM.L2W_Log_Mean_NCDS_Fin_EL} formulates Euler-Lagrange equations.

\begin{dfn}\label{DFN.L2W_Log_Mean_NCDS_Fin_EL_II}
Let $\mu\in\vartheta(\xi)$.

\begin{itemize}
\item[1)] For all $x\in A_{\xi}$ and $u\in B_{\xi}$, set

\begin{align}\label{EQ.DFN.L2W_Log_Mean_NCDS_Fin_EL_II_1}
\Lambda_{\mu}(x,u):=\Lambda_{\mu}^{\phi}(x,u)+\Lambda_{\mu}^{\bpsi}(x,u)\in B_{\xi}   
\end{align}

\begin{reapply}
\end{reapply}

\noindent using

\begin{align*}
\Lambda_{\mu}^{\phi}(x,u) &:= \int_{0}^{\infty}\lc\phi\lc\sharp\mu\rc{}+\alpha 1_{B}\rc^{-1}\phi(x)\lc\phi\lc\sharp\mu\rc{}+\alpha 1_{B}\rc^{-1}u\lc\bpsi\lc\sharp\mu\rc{}+\alpha 1_{B}\rc^{-1}d\alpha, \\
& \\
\Lambda_{\mu}^{\bpsi}(x,u) &:= \int_{0}^{\infty}\lc\phi\lc\sharp\mu\rc{}+\alpha 1_{B}\rc^{-1}u\lc\bpsi\lc\sharp\mu\rc{}+\alpha 1_{B}\rc^{-1}\bpsi(x)\lc\bpsi\lc\sharp\mu\rc{}+\alpha 1_{B}\rc^{-1}d\alpha.
\end{align*}

\begin{reapply}
\end{reapply}

\item[2)] For all $u,v\in B_{\xi}$, set

\begin{align}\label{EQ.DFN.L2W_Log_Mean_NCDS_Fin_EL_II_2}
\Lambda_{\mu}^{*}(u,v):=\Lambda_{\mu}^{\phi,*}(u,v)+\Lambda_{\mu}^{\bpsi,*}(u,v)\in A_{\xi} 
\end{align}

\begin{reapply}
\end{reapply}

\noindent using

\begin{align*}
\Lambda_{\mu}^{\phi,*}(u,v) &:= \phi^{*}\lc\int_{0}^{\infty}\lc\phi\lc\sharp\mu\rc{}+\alpha 1_{B}\rc^{-1}v\lc\bpsi\lc\sharp\mu\rc{}+\alpha 1_{B}\rc^{-1}u^{*}\lc\phi\lc\sharp\mu\rc{}+\alpha 1_{B}\rc^{-1}d\alpha\rc{}, \\
& \\
\Lambda_{\mu}^{\bpsi,*}(u,v) &:= \bpsi^{*}\lc\int_{0}^{\infty}\lc\bpsi\lc\sharp\mu\rc{}+\alpha 1_{B}\rc^{-1}u^{*}\lc\phi\lc\sharp\mu\rc{}+\alpha 1_{B}\rc^{-1}v\lc\bpsi\lc\sharp\mu\rc{}+\alpha 1_{B}\rc^{-1}d\alpha\rc{}.
\end{align*}

\begin{reapply}
\end{reapply}

\end{itemize}
\end{dfn}

\begin{prp}\label{PRP.Differential_Inverse}
Let $V$ be a unital Banach $^{*}$-algebra. If a map $F:[a,b]\longrightarrow\GL(V)$ is Fr\'echet differentiable in an open neighbourhood of $t_{0}\in (a,b)$, then

\begin{align}\label{EQ.PRP.Differential_Inverse_1}
\restr{0.925}{\frac{d}{dt}}{t=t_{0},\|.\|_{V}}F(t)^{-1}=-F(t_{0})^{-1}\cdot \restr{0.925}{\frac{d}{dt}}{t=t_{0},\|.\|_{V}}F(t)\cdot F(t_{0})^{-1}.
\end{align}
\end{prp}
\begin{proof}
Let $F:[a,b]\longrightarrow\GL(V)$ be Fr\'echet differentiable in an open neighbourhood of $t_{0}\in (a,b)$. The Leibniz rule lets us calculate

\begin{align}\label{EQ.PRP.Differential_Inverse_2}
0=\restr{0.925}{\frac{d}{dt}}{t=t_{0},\|.\|_{V}}F(t)F(t)^{-1}=\restr{0.925}{\frac{d}{dt}}{t=t_{0},\|.\|_{V}}F(t)\cdot F(t_{0})^{-1}+F(t_{0})\cdot \restr{0.925}{\frac{d}{dt}}{t=t_{0},\|.\|_{V}}F(t)^{-1}.
\end{align}

\noindent Equation \ref{EQ.PRP.Differential_Inverse_1} follows by solving Equation \ref{EQ.PRP.Differential_Inverse_2} for $\restr{0.925}{\frac{d}{dt}}{t=t_{0},\|.\|_{V}}F(t)^{-1}$.
\end{proof}

\begin{prp}\label{PRP.L2W_Log_Mean_NCDS_Fin_EL_II}
For all $\mu\in\vartheta(\xi)$, $x\in I(\Delta_{\xi})$ and $u,v\in B_{\xi}$, we have

\begin{itemize}
\item[1)] $\lgl\Lambda_{\mu}(x,u),v\rgl_{\omega}=\lgl x,\Lambda_{\mu}^{*}(u,v)\rgl_{\tau}$,

\item[2)] $d_{\mu}\mathcal{D}_{\xi}(x^{\flat})(u)=-\Lambda_{\mu}(x,u)$.
\end{itemize}
\end{prp}
\begin{proof}
For all $\mu\in\vartheta(\xi)$, $x\in I(\Delta_{\xi})$ and $u,v\in B_{\xi}$, we directly verify 

\begin{align}\label{EQ.PRP.L2W_Log_Mean_NCDS_Fin_EL_II_1}
\lgl \Lambda_{\mu}^{\phi}(x,u),v\rgl_{\omega}=\lgl x,\Lambda_{\mu}^{\phi,*}(u,v)\rgl_{\tau},\ \lgl \Lambda_{\mu}^{\bpsi}(x,u),v\rgl_{\omega}=\lgl x,\Lambda_{\mu}^{\bpsi,*}(u,v)\rgl_{\tau}.
\end{align}

\noindent Using Equation \ref{EQ.DFN.L2W_Log_Mean_NCDS_Fin_EL_II_1} and Equation \ref{EQ.DFN.L2W_Log_Mean_NCDS_Fin_EL_II_2}, note Equation \ref{EQ.PRP.L2W_Log_Mean_NCDS_Fin_EL_II_1} implies $1)$ by definition.\par
We show $2)$. Let $\mu\in\vartheta(\xi)$, $x\in I(\Delta_{\xi})$ and $u\in B_{\xi}$. Let $\varepsilon>0$ and $\mu:\lc{}-\varepsilon,\varepsilon\rc\longrightarrow\vartheta(\xi)$ smooth map s.t.~$\mu(0)=\mu$ and $\dot{\mu}(0)=x^{\flat}$. Then $2)$ in Proposition \ref{PRP.L2W_Log_Mean_NCDS_Fin_EL_I} shows

\begin{align}\label{EQ.PRP.L2W_Log_Mean_NCDS_Fin_EL_II_2}
\mathcal{D}_{\sharp\mu(t),\xi}(u)=\int_{0}^{\infty}\lc\phi\lc\sharp\mu(t)\rc{}+\alpha 1_{B}\rc^{-1}u\lc\bpsi\lc\sharp\mu(t)\rc{}+\alpha 1_{B}\rc^{-1}d\alpha
\end{align}

\noindent for all $t\in \lc{}-\varepsilon,\varepsilon\rc$. Using $d_{\mu}\mathcal{D}_{\xi}(x^{\flat})(u)=\restr{0.925}{\frac{d}{dt}}{t=0,\|.\|_{B}}\mathcal{D}_{\sharp\mu(t),\xi}(u)$ and further applying Fr\'echet derivative to the integrand in Equation \ref{EQ.PRP.L2W_Log_Mean_NCDS_Fin_EL_II_2}, the Leibniz rule lets us calculate

\begin{align*}
\restr{0.925}{\frac{d}{dt}}{t=0,\|.\|_{B}}\mathcal{D}_{\sharp\mu(t),\xi}(u) & = \int_{0}^{\infty}\restr{0.925}{\frac{d}{dt}}{t=0,\|.\|_{B}}\lc\phi\lc\sharp\mu(t)\rc{}+\alpha 1_{B}\rc^{-1}u\lc\bpsi\lc\sharp\mu\rc{}+\alpha 1_{B}\rc^{-1}d\alpha \phantom{\Bigg)} \\
& +\int_{0}^{\infty}\lc\phi\lc\sharp\mu\rc{}+\alpha 1_{B}\rc^{-1}u\restr{0.925}{\frac{d}{dt}}{t=0,\|.\|_{B}}\lc\bpsi\lc\sharp\mu(t)\rc{}+\alpha 1_{B}\rc^{-1}d\alpha. \phantom{\Bigg)}
\end{align*}

The above is the integral characterisation of $\restr{0.925}{\frac{d}{dt}}{t=0,\|.\|_{B}}\mathcal{D}_{\sharp\mu(t),\xi}(u)$. Proposition \ref{PRP.Differential_Inverse} further implies

\begin{align}\label{EQ.PRP.L2W_Log_Mean_NCDS_Fin_EL_II_3}
\restr{0.925}{\frac{d}{dt}}{t=0,\|.\|_{B}}\lc\phi\lc\sharp\mu(t)\rc{}+\alpha 1_{B}\rc^{-1}=-\lc\phi\lc\sharp\mu\rc{}+\alpha 1_{B}\rc^{-1}\phi(x)\lc\phi\lc\sharp\mu\rc{}+\alpha 1_{B}\rc^{-1}
\end{align}

\noindent and

\begin{align}\label{EQ.PRP.L2W_Log_Mean_NCDS_Fin_EL_II_4}
\restr{0.925}{\frac{d}{dt}}{t=0,\|.\|_{B}}\lc\bpsi\lc\sharp\mu(t)\rc{}+\alpha 1_{B}\rc^{-1}=-\lc\bpsi\lc\sharp\mu\rc{}+\alpha 1_{B}\rc^{-1}\bpsi(x)\lc\bpsi\lc\sharp\mu\rc{}+\alpha 1_{B}\rc^{-1}
\end{align}

\noindent for all $\alpha>0$. Equation \ref{EQ.PRP.L2W_Log_Mean_NCDS_Fin_EL_II_3} lets us calculate

\begin{align}\label{EQ.PRP.L2W_Log_Mean_NCDS_Fin_EL_II_5}
\int_{0}^{\infty}\restr{0.925}{\frac{d}{dt}}{t=0,\|.\|_{B}}\lc\phi\lc\sharp\mu(t)\rc{}+\alpha 1_{B}\rc^{-1}u\lc\bpsi\lc\sharp\mu\rc{}+\alpha 1_{B}\rc^{-1}d\alpha=-\Lambda_{\mu}^{\phi}(x,u),
\end{align}

\noindent whereas Equation \ref{EQ.PRP.L2W_Log_Mean_NCDS_Fin_EL_II_4} lets us calculate

\begin{align}\label{EQ.PRP.L2W_Log_Mean_NCDS_Fin_EL_II_6}
\int_{0}^{\infty}\lc\phi\lc\sharp\mu\rc{}+\alpha 1_{B}\rc^{-1}u\restr{0.925}{\frac{d}{dt}}{t=0,\|.\|_{B}}\lc\bpsi\lc\sharp\mu(t)\rc{}+\alpha 1_{B}\rc^{-1}d\alpha=-\Lambda_{\mu}^{\bpsi}(x,u).
\end{align}


\pagebreak


Using Equation \ref{EQ.DFN.L2W_Log_Mean_NCDS_Fin_EL_II_1}, applying Equation \ref{EQ.PRP.L2W_Log_Mean_NCDS_Fin_EL_II_5} and Equation \ref{EQ.PRP.L2W_Log_Mean_NCDS_Fin_EL_II_6} to the integral characterisation of $\restr{0.925}{\frac{d}{dt}}{t=0,\|.\|_{B}}\mathcal{D}_{\sharp\mu(t),\xi}(u)$ yields

\begin{align}\label{EQ.PRP.L2W_Log_Mean_NCDS_Fin_EL_II_7}
\restr{0.925}{\frac{d}{dt}}{t=0,\|.\|_{B}}\mathcal{D}_{\sharp\mu(t),\xi}(u)=-\lc\Lambda_{\mu}^{\phi}(x,u)+\Lambda_{\mu}^{\bpsi}(x,u)\rc{}=-\Lambda_{\mu}(x,u).
\end{align}

\noindent Equation \ref{EQ.PRP.L2W_Log_Mean_NCDS_Fin_EL_II_7} shows $2)$ at once.
\end{proof}

\begin{lem}\label{LEM.L2W_Log_Mean_NCDS_Fin_EL}
For all $\mu\in\vartheta(\xi)$ and $x,y,z\in I(\Delta_{\xi})$, we have

\begin{align}\label{EQ.LEM.L2W_Log_Mean_NCDS_Fin_EL_1}
\lgl d_{\mu}\mathfrak{F}^{-1}(x^{\flat})(y),z\rgl_{\tau}=-\lgl x,\Lambda_{\mu}^{*}\big(\Theta\big(\mu,y^{\flat}\big),\Theta\big(\mu,z^{\flat}\big)\big)\rgl_{\tau}.
\end{align}
\end{lem}
\begin{proof}
Let $\mu\in\vartheta(\xi)$ and $x,y,z\in I(\Delta_{\xi})$. We calculate $ d_{\mu}\mathfrak{F}^{-1}(x^{\flat})(y)$. Note

\begin{align}\label{EQ.LEM.L2W_Log_Mean_NCDS_Fin_EL_2}
\mathrlap{\phantom{\mathfrak{F}}_{\mu}}\mathfrak{F}^{-1}=\lc\nabla^{*}\mathcal{M}_{\sharp\mu,\xi}\nabla\rc^{-1},\ \mathcal{M}_{\sharp\mu,\xi}=\mathcal{D}_{\sharp\mu,\xi}^{-1}\in\GL(\BII(B_{\xi})).
\end{align}

\noindent Using the first identity in Equation \ref{EQ.LEM.L2W_Log_Mean_NCDS_Fin_EL_2}, Proposition \ref{PRP.Differential_Inverse} implies

\begin{align}\label{EQ.LEM.L2W_Log_Mean_NCDS_Fin_EL_3}
d_{\mu}\mathfrak{F}^{-1}(x^{\flat})(y)=-\mathrlap{\phantom{\mathfrak{F}}_{\mu}}\mathfrak{F}^{-1}\lc{}d_{\mu}\lc\nabla^{*}\mathcal{M}_{\xi}\nabla\rc{}(x^{\flat})\lc\mathrlap{\phantom{\mathfrak{F}}_{\mu}}\mathfrak{F}^{-1}(y)\rc\rc{}.
\end{align}

\noindent Since $\nabla$ and $\nabla^{*}$ are bounded linear, get $d_{\mu}\lc\nabla^{*}\mathcal{M}_{\xi}\nabla\rc{}(x^{\flat})=\nabla^{*}d_{\mu}\mathcal{M}_{\xi}(x^{\flat})\nabla$. Applying the latter to Equation \ref{EQ.LEM.L2W_Log_Mean_NCDS_Fin_EL_3} yields

\begin{align}\label{EQ.LEM.L2W_Log_Mean_NCDS_Fin_EL_4}
d_{\mu}\mathfrak{F}^{-1}(x^{\flat})(y)=-\mathrlap{\phantom{\mathfrak{F}}_{\mu}}\mathfrak{F}^{-1}\lc\nabla^{*}d_{\mu}\mathcal{M}_{\xi}(x^{\flat})\lc\nabla\mathrlap{\phantom{\mathfrak{F}}_{\mu}}\mathfrak{F}^{-1}(y)\rc\rc{}.
\end{align}

We therefore calculate $d_{\mu}\mathcal{M}_{\xi}(x^{\flat})\lc\nabla\mathrlap{\phantom{\mathfrak{F}}_{\mu}}\mathfrak{F}^{-1}(y)\rc$ in order to calculate $d_{\mu}\mathfrak{F}^{-1}(x^{\flat})(y)$. Using the second identity in Equation \ref{EQ.LEM.L2W_Log_Mean_NCDS_Fin_EL_2}, Proposition \ref{PRP.Differential_Inverse} implies

\begin{align}\label{EQ.LEM.L2W_Log_Mean_NCDS_Fin_EL_5}
d_{\mu}\mathcal{M}_{\xi}(x^{\flat})=-\mathcal{M}_{\sharp\mu,\xi}d_{\mu}\mathcal{D}_{\xi}(x^{\flat})\mathcal{M}_{\sharp\mu,\xi}.
\end{align}

\noindent Applying $2)$ in Proposition \ref{PRP.L2W_Log_Mean_NCDS_Fin_EL_II} for $u=\mathcal{M}_{\sharp\mu,\xi}\lc\nabla\mathrlap{\phantom{\mathfrak{F}}_{\mu}}\mathfrak{F}^{-1}(y)\rc$ to Equation \ref{EQ.LEM.L2W_Log_Mean_NCDS_Fin_EL_5} evaluated on $\nabla\mathrlap{\phantom{\mathfrak{F}}_{\mu}}\mathfrak{F}^{-1}(y)$ yields

\begin{align*}
d_{\mu}\mathcal{M}_{\xi}(x^{\flat})\lc\nabla\mathrlap{\phantom{\mathfrak{F}}_{\mu}}\mathfrak{F}^{-1}(y)\rc{} & = -\mathcal{M}_{\sharp\mu,\xi}\lc{}d_{\mu}\mathcal{D}_{\xi}(x^{\flat})\lc\mathcal{M}_{\sharp\mu,\xi}\lc\mathrlap{\phantom{\mathfrak{F}}_{\mu}}\mathfrak{F}^{-1}(y)\rc\rc\rc{} \phantom{\bigg)} \\
& =\mathcal{M}_{\sharp\mu,\xi}\lc\Lambda_{\mu}\lc{}x,\mathcal{M}_{\sharp\mu,\xi}\lc\nabla\mathrlap{\phantom{\mathfrak{F}}_{\mu}}\mathfrak{F}^{-1}(y)\rc\rc\rc{}. \phantom{\bigg)}
\end{align*}

\noindent Using $\mathcal{M}_{\sharp\mu,\xi}\lc\nabla\mathrlap{\phantom{\mathfrak{F}}_{\mu}}\mathfrak{F}^{-1}(y)\rc{}=\sharp\Theta\big(\mu,y^{\flat}\big)$, we therefore obtain

\begin{align}\label{EQ.LEM.L2W_Log_Mean_NCDS_Fin_EL_6}
d_{\mu}\mathcal{M}_{\xi}(x^{\flat})\lc\nabla\mathrlap{\phantom{\mathfrak{F}}_{\mu}}\mathfrak{F}^{-1}(y)\rc{}=\mathcal{M}_{\sharp\mu,\xi}\lc\Lambda_{\mu}\big(x,\sharp\Theta\big(\mu,y^{\flat}\big)\big)\rc{}.
\end{align}


\pagebreak


Equation \ref{EQ.LEM.L2W_Log_Mean_NCDS_Fin_EL_4} and Equation \ref{EQ.LEM.L2W_Log_Mean_NCDS_Fin_EL_6} show

\begin{align}\label{EQ.LEM.L2W_Log_Mean_NCDS_Fin_EL_7}
d_{\mu}\mathfrak{F}^{-1}(x^{\flat})(y)=-\mathrlap{\phantom{\mathfrak{F}}_{\mu}}\mathfrak{F}^{-1}\lc\nabla^{*}\mathcal{M}_{\sharp\mu,\xi}\lc\Lambda_{\mu}\big(x,\sharp\Theta\big(\mu,y^{\flat}\big)\big)\rc\rc{}.
\end{align}

\noindent We show Equation \ref{EQ.LEM.L2W_Log_Mean_NCDS_Fin_EL_1}. Equation \ref{EQ.LEM.L2W_Log_Mean_NCDS_Fin_EL_7}, together with $1)$ in Proposition \ref{PRP.L2W_Log_Mean_NCDS_Fin_EL_II} applied to the third identity in our below calculation, lets us calculate

\begin{align*}
\lgl d_{\mu}\mathfrak{F}^{-1}(x^{\flat})(y),z\rgl_{\tau} & = -\lgl-\mathrlap{\phantom{\mathfrak{F}}_{\mu}}\mathfrak{F}^{-1}\lc\nabla^{*}\mathcal{M}_{\sharp\mu,\xi}\lc\Lambda_{\mu}\big(x,\sharp\Theta\big(\mu,y^{\flat}\big)\big)\rc\rc{},z\rgl_{\omega} \phantom{\bigg)} \\
& = -\lgl\Lambda_{\mu}\big(x,\sharp\Theta\big(\mu,y^{\flat}\big)\big),\sharp\Theta\big(\mu,z^{\flat}\big)\rgl_{\omega} \phantom{\bigg)} \\
& = -\lgl x,\Lambda_{\mu}^{*}\big(\sharp\Theta\big(\mu,y^{\flat}\big),\sharp\Theta\big(\mu,z^{\flat}\big)\big)\rgl_{\tau}. \phantom{\bigg)}
\end{align*}

\noindent The above calculation shows Equation \ref{EQ.LEM.L2W_Log_Mean_NCDS_Fin_EL_1}.
\end{proof}

\begin{thm}\label{THM.L2W_Log_Mean_NCDS_Fin_EL}
Let $(\phi,\bpsi,\gamma,\nabla)$ be noncommutative differential structure for tracial AF-$C^{*}$-algebras $(A,\tau)$ and $(B,\omega)$ in the finite-dimensional logarithmic mean setting. Let $\xi\in\SII(A)$ be a fixed state. A smooth path $\mu:[0,1]\longrightarrow\vartheta(\xi)$ satisfies the Euler-Lagrange equations of the energy functional induced by $g^{\xi}$ if and only if

\begin{align}\label{EQ.THM.L2W_Log_Mean_NCDS_Fin_EL_1}
\restr{0.925}{\frac{d}{ds}}{s=t,\|.\|_{A^{*}}}\mathfrak{F}_{\mu(s)}^{-1}\lc\dot{\mu}(s)\rc{}=-\frac{1}{2}\Lambda_{\mu(t)}^{*}\lc\sharp\Theta\lc\mu(t),\dot{\mu}(t)\rc{},\sharp\Theta\lc\mu(t),\dot{\mu}(t)\rc\rc{}
\end{align}

\noindent for all $t\in (0,1)$.
\end{thm}
\begin{proof}
We consider first variation of energy \cite{BK.Lan.1995.Riemannian_Manifolds}. Note $T\vartheta(\xi)=\vartheta(\xi)\times I(\Delta_{\xi})^{\flat}$ is trivial by $2)$ in Proposition \ref{PRP.RM_Embedded_Submanifold_II}. It suffices to solve for critical points of the energy functional induced by $g^{\xi}$ on variations of form $\mu\lc{}t,\varepsilon\rc{}=\mu(t)+\varepsilon\eta(t)$ using $\eta\in C_{0}^{\infty}\lc{}[0,1],I(\Delta_{\xi})^{\flat}\rc$ and $\varepsilon\in \lc{}-\delta,\delta\rc$ for $\delta>0$ sufficiently small. The latter is chosen s.t.~$\mu\lc{}t,\varepsilon\rc\in\vartheta(\xi)$ in each case.\par
Let $\mu\lc{}t,\varepsilon\rc{}:=\mu(t)+\varepsilon\eta(t)$ be such a variation. Lemma \ref{LEM.L2W_Log_Mean_NCDS_Fin_EL} shows

\begin{align}\label{EQ.THM.L2W_Log_Mean_NCDS_Fin_EL_2}
\lgl d_{\mu(t)}\mathfrak{F}^{-1}\lc\eta(t)\rc\lc\sharp\dot{\mu}(t)\rc{},\sharp\dot{\mu}(t)\rgl_{\tau}=-\lgl\sharp\eta(t),\Lambda_{\mu(t)}^{*}\lc\sharp\Theta\lc\mu(t),\dot{\mu}(t)\rc{},\sharp\Theta\lc\mu(t),\dot{\mu}(t)\rc\rc\rgl_{\tau}
\end{align}

\noindent for all $t\in [0,1]$. Note $g^{\xi}\cong\mathfrak{F}^{-1}$ via GNS-inner product of $\tau$ restricted to $A_{\xi}$. Using the latter, we calculate

\begin{align*}
& \restr{0.925}{\frac{d}{d\varepsilon}}{\varepsilon=0}\int_{0}^{1}g_{\mu\lc{}t,\varepsilon\rc{}}^{\xi}\lc\dot{\mu}(t)+\varepsilon\dot{\eta}(t),\dot{\mu}(t)+\varepsilon\dot{\eta}(t)\rc{}dt \phantom{\Bigg)} \\
=& \int_{0}^{1}2g_{\mu(t)}^{\xi}\lc\dot{\mu}(t),\dot{\eta}(t)\rc{}+\lgl\restr{0.925}{\frac{d}{d\varepsilon}}{\varepsilon=0,\|.\|_{A^{*}}}\mathfrak{F}^{-1}\lc\mu\lc{}t,\varepsilon\rc\rc\lc\sharp\dot{\mu}(t)\rc{},\sharp\dot{\mu}(t)\rgl_{\tau}dt \phantom{\Bigg)} \\
=& \int_{0}^{1}2g_{\mu(t)}^{\xi}\lc\dot{\mu}(t),\dot{\eta}(t)\rc{}+\lgl d_{\mu(t)}\mathfrak{F}^{-1}\lc\eta(t)\rc\lc\sharp\dot{\mu}(t)\rc{},\sharp\dot{\mu}(t)\rgl_{\tau}dt. \phantom{\Bigg)}
\end{align*}

We thus apply Equation \ref{EQ.THM.L2W_Log_Mean_NCDS_Fin_EL_2}, symmetry of the real inner product and integration by parts in order to calculate

\begin{align*}
& \restr{0.925}{\frac{d}{d\varepsilon}}{\varepsilon=0,\|.\|_{A^{*}}}\int_{0}^{1}g_{\mu\lc{}t,\varepsilon\rc{}}^{\xi}\lc\dot{\mu}(t)+\varepsilon\dot{\eta}(t),\dot{\mu}(t)+\varepsilon\dot{\eta}(t)\rc{}dt \phantom{\Bigg)} \\
=& \int_{0}^{1}2g_{\mu(t)}^{\xi}\lc\dot{\mu}(t),\dot{\eta}(t)\rc{}-\lgl\Lambda_{\mu(t)}^{*}\lc\sharp\Theta\lc\mu(t),\dot{\mu}(t)\rc{},\sharp\Theta\lc\mu(t),\dot{\mu}(t)\rc\rc{},\sharp\eta(t)\rgl_{\tau}dt \phantom{\Bigg)} \\
=& -\lc\int_{0}^{1}2\lgl\restr{0.925}{\frac{d}{ds}}{s=t,\|.\|_{A^{*}}}\mathfrak{F}_{\mu(s)}^{-1}\lc\sharp\dot{\mu}(s)\rc{},\sharp\eta(t)\rgl_{\tau}+\lgl\Lambda_{\mu(t)}^{*}\lc\sharp\Theta\lc\mu(t),\dot{\mu}(t)\rc{},\sharp\Theta\lc\mu(t),\dot{\mu}(t)\rc\rc{},\sharp\eta(t)\rgl_{\tau}dt\rc. \phantom{\Bigg)}
\end{align*}

\noindent We solve for critical points of the energy functional induced by $g^{\xi}$. Using the formula for first variation of energy \lc{}cf.~proof of Theorem IX.4.3 in \cite{BK.Lan.1995.Riemannian_Manifolds}\rc{}, the above calculation hence shows Equation \ref{EQ.THM.L2W_Log_Mean_NCDS_Fin_EL_1} gives Euler-Lagrange equations.
\end{proof}


\subsubsection*{Hessians of quantum relative entropy}

Theorem \ref{THM.L2W_Log_Mean_NCDS_Fin_Hessian} gives two differential equations for Hessians of quantum relative entropy used in Lemma \ref{LEM.L2W_EVI_Equivalence}, i.e.~required for our equivalence theorem. Let $(\phi,\bpsi,\gamma,\nabla)$ be noncommutative differential structure for tracial AF-$C^{*}$-algebras $(A,\tau)$ and $(B,\omega)$ in the finite-dimensional logarithmic mean setting. Let $\xi\in\SII(A)$ be a fixed state. Finite-dimensionality ensures finite support.\par
Note $2.2)$ in Theorem \ref{THM.Wstar_Derivation_QG_HSG_Regularity} and $1)$ in Corollary \ref{COR.L2W_Log_Mean_NCDS} imply smoothness of quantum relative entropy restricted to relative interiors. Proposition \ref{PRP.L2W_Log_Mean_NCDS_Fin_Hessian} expresses Hessians of quantum relative entropy in terms of $\Lambda$-operations.

\begin{ntn}
Let $\Hess\Enttau$ denote the Hessian of $\Enttau$ restricted to $\vartheta(\xi)$. We write $\Hess\hspace{-0.1525cm} \phantom{.}_{\mu}\Enttau(\eta):=\Hess\hspace{-0.1525cm} \phantom{.}_{\mu}\Enttau(\eta,\eta)$ upon evaluation at $\mu\in\vartheta(\xi)$ and $(\eta,\eta)\in T_{\mu}\vartheta(\xi)^{2}$.
\end{ntn}

\begin{prp}\label{PRP.L2W_Log_Mean_NCDS_Fin_Hessian}
For all $\mu\in\vartheta(\xi)$ and $\eta\in I(\Delta_{\xi})^{\flat}$, we have 

\begin{align}\label{EQ.PRP.L2W_Log_Mean_NCDS_Fin_Hessian_1}
\Hess\hspace{-0.1525cm} \phantom{.}_{\mu}\Enttau(\eta)=-\frac{1}{2}\lgl\Lambda_{\mu}^{*}\lc\sharp\Theta(\mu,\eta),\sharp\Theta(\mu,\eta)\rc{},\Delta\sharp\mu\rgl_{\tau}+g_{\mu}^{\xi}\big(\eta,\lc\Delta\sharp\eta\rc^{\flat}\big).
\end{align}
\end{prp}
\begin{proof}
Let $\mu\in\vartheta(\xi)$ and $\eta\in I(\Delta_{\xi})^{\flat}$. Note $\mathfrak{F}_{\mu}=\nabla^{*}\mathcal{M}_{\sharp\mu,\xi}\nabla$ commutes with adjoining. Using $\log\sharp\mu\in A_{\xi}$, Proposition \ref{PRP.L2W_Log_Mean_NCDS} lets us calculate

\begin{align}\label{EQ.PRP.L2W_Log_Mean_NCDS_Fin_Hessian_2}
\tau\lc\sharp\eta\log\sharp\mu\rc{}=\lgl\mathrlap{\phantom{\mathfrak{F}}_{\mu}}\mathfrak{F}^{-1}\lc\sharp\eta\rc{},\nabla^{*}\mathcal{M}_{\sharp\mu,\xi}\nabla\log\sharp\mu\rgl_{\tau}=\lgl\mathrlap{\phantom{\mathfrak{F}}_{\mu}}\mathfrak{F}^{-1}\lc\sharp\eta\rc{},\Delta\sharp\mu\rgl_{\tau}=\tau\lc\mathrlap{\phantom{\mathfrak{F}}_{\mu}}\mathfrak{F}^{-1}\lc\sharp\eta\rc\Delta\sharp\mu\rc{}.
\end{align}

\noindent Using $1)$ in Lemma \ref{LEM.L2W_Log_Mean_NCDS}, Equation \ref{EQ.PRP.L2W_Log_Mean_NCDS_Fin_Hessian_2} implies

\begin{align}\label{EQ.PRP.L2W_Log_Mean_NCDS_Fin_Hessian_3}
\frac{d}{dt}\Enttau\lc\mu(t)\rc{}=\tau\lc\sharp\dot{\mu}(t)\log\sharp\mu(t)\rc{}=\tau\lc\mathfrak{F}_{\mu(t)}^{-1}\lc\sharp\dot{\mu}(t)\rc\Delta\sharp\mu(t)\rc{}
\end{align}

\noindent for all smooth paths $\mu:[a,b]\longrightarrow\vartheta(\xi)$.\par
Let $\mu:[0,1]\longrightarrow\vartheta(\xi)$ be a geodesic s.t.~$\mu=\mu(0)$ and $\dot{\mu}(0)=\eta$. Using the chain rule of Riemannian metrics involving covariant derivatives \cite{BK.Lan.1995.Riemannian_Manifolds}, we argue as \cite{ART.Ott.2005.Classical_OT_GradFlow_DisConvex} to get

\begin{align}\label{EQ.PRP.L2W_Log_Mean_NCDS_Fin_Hessian_4}
\Hess\hspace{-0.1525cm} \phantom{.}_{\mu}\Enttau(\eta) = \restr{0.925}{\frac{d^{2}}{dt^{2}}}{t=0}\Enttau\lc\mu(t)\rc{}.
\end{align}

\noindent Equation \ref{EQ.PRP.L2W_Log_Mean_NCDS_Fin_Hessian_3} and Equation \ref{EQ.PRP.L2W_Log_Mean_NCDS_Fin_Hessian_4} let us calculate

\begin{align}\label{EQ.PRP.L2W_Log_Mean_NCDS_Fin_Hessian_5}
\Hess\hspace{-0.1525cm} \phantom{.}_{\mu}\Enttau(\eta)=\restr{0.925}{\frac{d^{2}}{dt^{2}}}{t=0}\tau\lc\sharp\mu(t)\log\sharp\mu(t)\rc{}=\restr{0.925}{\frac{d}{dt}}{t=0}\tau\lc\mathfrak{F}_{\mu(t)}^{-1}\lc\sharp\dot{\mu}(t)\rc\Delta\sharp\mu(t)\rc{}.
\end{align}

We show Equation \ref{EQ.PRP.L2W_Log_Mean_NCDS_Fin_Hessian_3}. All geodesics are critical points of the energy functional induced by $g^{\xi}$ \cite{BK.Lan.1995.Riemannian_Manifolds}. Using Equation \ref{EQ.PRP.L2W_Log_Mean_NCDS_Fin_Hessian_5} for the first and Theorem \ref{THM.L2W_Log_Mean_NCDS_Fin_EL} for the third identity in our below calculation, we therefore calculate

\begin{align*}
\Hess\hspace{-0.1525cm} \phantom{.}_{\mu}\Enttau(\eta) & = \restr{0.925}{\frac{d}{dt}}{t=0}\tau\lc\mathfrak{F}_{\mu(t)}^{-1}\lc\sharp\dot{\mu}(t)\rc\Delta\sharp\mu(t)\rc{} \phantom{\Bigg)} \\
& = \lgl\restr{0.925}{\frac{d}{dt}}{t=0}\mathfrak{F}_{\mu(t)}^{-1}\lc\sharp\dot{\mu}(t)\rc{},\Delta\sharp\mu\rgl_{\tau}+g_{\mu}^{\xi}\big(\eta,\lc\Delta\sharp\eta\rc^{\flat}\big) \phantom{\Bigg)} \\
& = -\frac{1}{2}\lgl\Lambda_{\mu}^{*}\big(\Theta\big(\mu,x^{\flat}\big),\Theta\big(\mu,x^{\flat}\big)\big),\Delta\mu\rgl_{\tau}+g_{\mu}^{\xi}\big(\eta,\lc\Delta\sharp\eta\rc^{\flat}\big). \phantom{\Bigg)}
\end{align*}

\noindent The above calculation shows Equation \ref{EQ.PRP.L2W_Log_Mean_NCDS_Fin_Hessian_3}.
\end{proof}

\begin{thm}\label{THM.L2W_Log_Mean_NCDS_Fin_Hessian}
Let $(\phi,\bpsi,\gamma,\nabla)$ be noncommutative differential structure for tracial AF-$C^{*}$-algebras $(A,\tau)$ and $(B,\omega)$ in the finite-dimensional logarithmic mean setting. Let\linebreak $\xi\in\SII(A)$ be a fixed state. Let $\varphi_{0},\varphi_{1}:U\longrightarrow (0,\infty)$ be smooth maps for $U\subset (0,\infty)\times (0,\infty)$ open.

\begin{itemize}
\item[1)] Assume $U=(0,\infty)\times (0,\infty)$ and $\varphi:=\varphi_{0}=\varphi_{1}$. Let $\mu:[0,1]\longrightarrow\vartheta(\xi)$ be smooth. Using the latter, we define smooth map $\eta:U\longrightarrow\vartheta(\xi)$ by setting

\begin{align}\label{EQ.THM.L2W_Log_Mean_NCDS_Fin_Hessian_1}
\eta(t,s):=h_{\varphi(t,s)}\lc\mu(t)\rc{}
\end{align}

\begin{reapply}
\end{reapply}

\noindent for all $t,s>0$. We have

\begin{align}\label{EQ.THM.L2W_Log_Mean_NCDS_Fin_Hessian_2}
\frac{1}{2}\frac{\partial}{\partial s}g_{\eta}^{\xi}\lc\frac{\partial}{\partial t}\eta,\frac{\partial}{\partial t}\eta\rc{}+\frac{\partial^{2}}{\partial s\partial t}\varphi\cdot \frac{\partial}{\partial t}\Enttau(\eta)=-\frac{\partial}{\partial s}\varphi\cdot \Hess\hspace{-0.1525cm} \phantom{.}_{\eta}\Enttau\lc\frac{\partial}{\partial t}\eta\rc{}
\end{align}

\begin{reapply}
\end{reapply}

\noindent on $(0,\infty)\times (0,\infty)$.

\item[2)] Assume $\frac{\partial}{\partial s}\varphi_{0}=-\frac{\partial}{\partial s}\varphi_{1}$. Let $\mu\in\vartheta(\xi)$ and $x\in I(\Delta_{\xi})$. Using the latter, we define smooth maps $\eta:U\longrightarrow\vartheta(\xi)$ and $X:U\longrightarrow I(\Delta_{\xi})$ by setting 

\begin{align}\label{EQ.THM.L2W_Log_Mean_NCDS_Fin_Hessian_3}
\eta(t,s):=h_{\varphi_{0}(t,s)}(\mu),\ X(t,s):=h_{\varphi_{1}(t,s)}(x)
\end{align}

\begin{reapply}
\end{reapply}

\noindent for all $(t,s)\in U$. We have

\begin{align}\label{EQ.THM.L2W_Log_Mean_NCDS_Fin_Hessian_4}
\frac{1}{2}\frac{\partial}{\partial s}\dblv{}\mathcal{M}_{\eta}^{\frac{1}{2}}\nabla X\dblv_{\omega}^{2}=-\frac{\partial}{\partial s}\varphi_{1}\cdot \Hess\hspace{-0.1525cm} \phantom{.}_{\eta}\Enttau\lc\mathfrak{F}_{\eta}(X)^{\flat}\rc
\end{align}

\begin{reapply}
\end{reapply}

\noindent on $U$.
\end{itemize}
\end{thm}
\begin{proof}
We show $1)$. Assume its setting. Note $\frac{\partial}{\partial s}\sharp\eta(t,s)=-\frac{\partial}{\partial s}\varphi(t,s)\cdot \Delta\sharp\eta(t,s)$ and further $\frac{\partial^{2}}{\partial s\partial t}\sharp\eta(t,s)=-\frac{\partial^{2}}{\partial s\partial t}\varphi(t,s)\cdot \Delta\sharp\eta(t,s)-\frac{\partial}{\partial s}\varphi(t,s)\cdot \Delta\frac{\partial}{\partial t}\sharp\eta(t,s)$. We calculate

\begin{align*}
\frac{1}{2}\frac{\partial}{\partial s}g_{\eta}^{\xi}\lc\frac{\partial}{\partial t}\eta,\frac{\partial}{\partial t}\eta\rc{}(t,s) =& -\frac{\partial}{\partial s}\varphi(t,s)\cdot \frac{1}{2}\bigg\langle d_{\eta(t,s)}\mathfrak{F}^{-1}\lc\lc\Delta\sharp\eta(t,s)\rc^{\flat}\rc\lc\frac{\partial}{\partial t}\sharp\eta(t,s)\rc{},\frac{\partial}{\partial t}\sharp\eta(t,s)\bigg\rangle_{\tau} \phantom{\Bigg)} \\
& +g_{\eta(t,s)}^{\xi}\lc\frac{\partial^{2}}{\partial s\partial t}\eta(t,s),\frac{\partial}{\partial t}\eta(t,s)\rc \phantom{\Bigg)} \\
=& -\frac{\partial}{\partial s}\varphi(t,s)\cdot \frac{1}{2}\bigg\langle d_{\eta(t,s)}\mathfrak{F}^{-1}\lc\lc\Delta\sharp\eta(t,s)\rc^{\flat}\rc\lc\frac{\partial}{\partial t}\sharp\eta(t,s)\rc{},\frac{\partial}{\partial t}\sharp\eta(t,s)\bigg\rangle_{\tau} \phantom{\Bigg)} \\
& -\frac{\partial^{2}}{\partial s\partial t}\varphi(t,s)\cdot g_{\eta(t,s)}^{\xi}\bigg(\lc\Delta\sharp\eta(t,s)\rc^{\flat},\frac{\partial}{\partial t}\eta(t,s)\bigg) \phantom{\Bigg)} \\
& -\frac{\partial}{\partial s}\varphi(t,s)\cdot g_{\eta(t,s)}^{\xi}\bigg(\big(\Delta\frac{\partial}{\partial t}\sharp\eta(t,s)\big)^{\flat},\frac{\partial}{\partial t}\eta(t,s)\bigg). \phantom{\Bigg)}
\end{align*}

\noindent Using $1)$ in Lemma \ref{LEM.L2W_Log_Mean_NCDS} and symmetry of the real inner product, we calculate

\begin{align*}
\frac{\partial}{\partial t}\Enttau\lc\eta(t,s)\rc{} & = \lgl\frac{\partial}{\partial t}\sharp\eta(t,s),\log\sharp\eta(t,s)\rgl_{\tau} \phantom{\Bigg)} \\
& = g_{\eta(t,s)}\lc\frac{\partial}{\partial t}\eta(t,s),\mathfrak{F}_{\eta(t,s)}\lc\log\sharp\eta(t,s)\rc^{\flat}\rc \phantom{\Bigg)} \\
& = g_{\eta(t,s)}^{\xi}\bigg(\lc\Delta\sharp\eta(t,s)\rc^{\flat},\frac{\partial}{\partial t}\eta(t,s)\bigg). \phantom{\Bigg)}
\end{align*}

\noindent We combine the two calculations above. We obtain

\begin{align*}
\frac{1}{2}\frac{\partial}{\partial s}g_{\eta}^{\xi}\lc\frac{\partial}{\partial t}\eta,\frac{\partial}{\partial t}\eta\rc{}(t,s) =& -\frac{\partial}{\partial s}\varphi(t,s)\cdot \frac{1}{2}\bigg\langle d_{\eta(t,s)}\mathfrak{F}^{-1}\lc\lc\Delta\sharp\eta(t,s)\rc^{\flat}\rc\lc\frac{\partial}{\partial t}\sharp\eta(t,s)\rc{},\frac{\partial}{\partial t}\sharp\eta(t,s)\bigg\rangle_{\tau} \phantom{\Bigg)} \\
& -\frac{\partial^{2}}{\partial s\partial t}\varphi(t,s)\cdot \frac{\partial}{\partial t}\Enttau\lc\eta(t,s)\rc \phantom{\Bigg)} \\
& -\frac{\partial}{\partial s}\varphi(t,s)\cdot g_{\eta(t,s)}^{\xi}\bigg(\big(\Delta\frac{\partial}{\partial t}\sharp\eta(t,s)\big)^{\flat},\frac{\partial}{\partial t}\eta(t,s)\bigg). \phantom{\Bigg)}
\end{align*}

We readily see adding $\frac{\partial^{2}}{\partial s\partial t}\varphi(t,s)\cdot \frac{\partial}{\partial t}\Enttau\lc\eta(t,s)\rc$ to both sides of the above identity shows Equation \ref{EQ.THM.L2W_Log_Mean_NCDS_Fin_Hessian_2} if

\begin{align*}
\Hess\hspace{-0.1525cm} \phantom{.}_{\eta}\Enttau\lc\frac{\partial}{\partial t}\eta\rc{}(t.s) =& \frac{1}{2}\bigg\langle d_{\eta(t,s)}\mathfrak{F}^{-1}\lc\lc\Delta\sharp\eta(t,s)\rc^{\flat}\rc\lc\frac{\partial}{\partial t}\sharp\eta(t,s)\rc{},\frac{\partial}{\partial t}\sharp\eta(t,s)\bigg\rangle_{\tau} \phantom{\Bigg)} \\
& +g_{\eta(t,s)}^{\xi}\bigg(\big(\Delta\frac{\partial}{\partial t}\sharp\eta(t,s)\big)^{\flat},\frac{\partial}{\partial t}\eta(t,s)\bigg) \phantom{\Bigg)}
\end{align*}

\noindent for all $t,s>0$. We show the above identity. Let $t,s>0$. Proposition \ref{PRP.L2W_Log_Mean_NCDS_Fin_Hessian} implies

\begin{align*}
\Hess\hspace{-0.1525cm} \phantom{.}_{\eta}\Enttau\lc\frac{\partial}{\partial t}\eta\rc{}(t.s) =& -\frac{1}{2}\bigg\langle\Lambda_{\eta(t,s)}^{*}\bigg(\sharp\Theta\bigg(\eta(t,s),\frac{\partial}{\partial t}\eta(t,s)\bigg),\sharp\Theta\bigg(\eta(t,s),\frac{\partial}{\partial t}\eta(t,s)\bigg)\bigg),\Delta\sharp\eta(t,s)\bigg\rangle_{\tau} \phantom{\Bigg)} \\
& +g_{\eta(t,s)}^{\xi}\lc\frac{\partial}{\partial t}\eta(t,s),\big(\Delta\frac{\partial}{\partial t}\sharp\eta(t,s)\big)^{\flat}\rc{}. \phantom{\Bigg)}
\end{align*}

\noindent Applying Lemma \ref{LEM.L2W_Log_Mean_NCDS_Fin_EL} to the first term and symmetry of the real inner product to the second one above, we obtain the claimed identity. Thus $\Hess\hspace{-0.1525cm} \phantom{.}_{\eta}\Enttau\big(\frac{\partial}{\partial t}\eta\big)$ is of required form, hence $1)$ holds.\par
We show $2)$. Assume its setting. For all $(t,s)\in U$, set $\mathfrak{F}_{\eta,X}(t,s):=\mathfrak{F}_{\eta(t,s)}\lc{}X(t,s)\rc$. Let $(t,s)\in U$. We have

\begin{align}\label{EQ.THM.L2W_Log_Mean_NCDS_Fin_Hessian_5}
\dblv{}\mathcal{M}_{\eta(t,s)}^{\frac{1}{2}}\nabla X(t,s)\dblv_{\omega}^{2}=\lgl\mathfrak{F}_{\eta,X}(t,s),X(t,s)\rgl_{\tau}.
\end{align}

\noindent Using Equation \ref{EQ.THM.L2W_Log_Mean_NCDS_Fin_Hessian_5}, the Leibniz rule lets us calculate

\begin{align}\label{EQ.THM.L2W_Log_Mean_NCDS_Fin_Hessian_6}
\frac{\partial}{\partial s}\dblv{}\mathcal{M}_{\eta(t,s)}^{\frac{1}{2}}\nabla X(t,s)\dblv_{\omega}^{2}=\lgl\frac{\partial}{\partial s}\mathfrak{F}_{\eta,X}(t,s),X(t,s)\rgl_{\tau}+2\lgl\mathfrak{F}_{\eta(t,s)}\lc\frac{\partial}{\partial s}X(t,s)\rc{},X(t,s)\rgl_{\tau}.
\end{align}

\noindent We therefore calculate the two summands on the right-hand side of Equation \ref{EQ.THM.L2W_Log_Mean_NCDS_Fin_Hessian_6} in order. Note $\frac{\partial}{\partial s}\eta(t,s)=-\frac{\partial}{\partial s}\varphi_{0}(t,s)\cdot \Delta\sharp\eta(t,s)$. Applying Proposition \ref{PRP.Differential_Inverse} to $\mathfrak{F}=\lc\mathfrak{F}^{-1}\rc^{-1}$ and further using $\frac{\partial}{\partial s}\varphi_{0}=-\frac{\partial}{\partial s}\varphi_{1}$, we calculate

\begin{align*}
\lgl\frac{\partial}{\partial s}\mathfrak{F}_{\eta,X}(t,s),X(t,s)\rgl_{\tau} & = -\lgl\frac{\partial}{\partial s}\mathfrak{F}_{\eta(t,s)}^{-1}\lc\mathfrak{F}_{\eta,X}(t,s)\rc{},\mathfrak{F}_{\eta,X}(t,s)\rgl_{\tau} \phantom{\Bigg)} \\
& = \frac{\partial}{\partial s}\varphi_{0}(t,s)\cdot \bigg\langle d_{\eta(t,s)}\mathfrak{F}^{-1}\lc\lc\Delta\sharp\eta(t,s)\rc^{\flat}\rc\lc\mathfrak{F}_{\eta,X}(t,s)\rc{},\mathfrak{F}_{\eta,X}(t,s)\bigg\rangle_{\tau} \phantom{\Bigg)} \\
& = -\frac{\partial}{\partial s}\varphi_{1}(t,s)\cdot \bigg\langle d_{\eta(t,s)}\mathfrak{F}^{-1}\lc\lc\Delta\sharp\eta(t,s)\rc^{\flat}\rc\lc\mathfrak{F}_{\eta,X}(t,s)\rc{},\mathfrak{F}_{\eta,X}(t,s)\bigg\rangle_{\tau}. \phantom{\Bigg)}
\end{align*}


\pagebreak


Using $\frac{\partial}{\partial s}X(t,s)=-\frac{\partial}{\partial s}\varphi_{1}(t,s)\cdot \Delta X(t,s)$, we moreover calculate

\begin{align*}
\lgl\mathfrak{F}_{\eta(t,s)}\lc\frac{\partial}{\partial s}X(t,s)\rc{},X(t,s)\rgl_{\tau} & = -\frac{\partial}{\partial s}\varphi_{1}(t,s)\cdot \lgl\mathfrak{F}_{\eta(t,s)}\lc\Delta X(t,s)\rc{},X(t,s)\rgl_{\tau} \phantom{\Bigg)} \\
& = -\frac{\partial}{\partial s}\varphi_{1}(t,s)\cdot \lgl X(t,s),\Delta\mathfrak{F}_{\eta,X}(t,s)\rgl_{\tau} \phantom{\Bigg)} \\
& = -\frac{\partial}{\partial s}\varphi_{1}(t,s)\cdot g_{\eta(t,s)}^{\xi}\lc\lc\mathfrak{F}_{\eta,X}(t,s)\rc^{\flat},\lc\Delta\mathfrak{F}_{\eta,X}(t,s)\rc^{\flat}\rc{}. \phantom{\Bigg)}
\end{align*}

\noindent We combine the two calculations above with Equation \ref{EQ.THM.L2W_Log_Mean_NCDS_Fin_Hessian_6}. We obtain

\begin{align*}
\frac{1}{2}\frac{\partial}{\partial s}\dblv{}\mathcal{M}_{\eta(t,s)}^{\frac{1}{2}}\nabla X(t,s)\dblv_{\omega}^{2} =& -\frac{\partial}{\partial s}\varphi_{1}(t,s)\Bigg(\frac{1}{2}\bigg\langle d_{\eta(t,s)}\mathfrak{F}^{-1}\lc\lc\Delta\sharp\eta(t,s)\rc^{\flat}\rc\lc\mathfrak{F}_{\eta,X}(t,s)\rc{},\mathfrak{F}_{\eta,X}(t,s)\bigg\rangle_{\tau} \phantom{\Bigg)} \\
& +g_{\eta(t,s)}^{\xi}\lc\lc\mathfrak{F}_{\eta,X}(t,s)\rc^{\flat},\lc\Delta\mathfrak{F}_{\eta,X}(t,s)\rc^{\flat}\rc\Bigg). \phantom{\Bigg)}
\end{align*}

\noindent We see the above identity implies Equation \ref{EQ.THM.L2W_Log_Mean_NCDS_Fin_Hessian_4} if

\begin{align*}
\Hess\hspace{-0.1525cm} \phantom{.}_{\eta(t,s)}\Enttau\lc\lc\mathfrak{F}_{\eta,X}(t,s)\rc^{\flat}\rc{} =& \frac{1}{2}\bigg\langle d_{\eta(t,s)}\mathfrak{F}^{-1}\lc\lc\Delta\sharp\eta(t,s)\rc^{\flat}\rc\lc\mathfrak{F}_{\eta,X}(t,s)\rc{},\mathfrak{F}_{\eta,X}(t,s)\bigg\rangle_{\tau} \phantom{\Bigg)} \\
& +g_{\eta(t,s)}^{\xi}\lc\lc\mathfrak{F}_{\eta,X}(t,s)\rc^{\flat},\lc\Delta\mathfrak{F}_{\eta,X}(t,s)\rc^{\flat}\rc \phantom{\Bigg)}
\end{align*}

\noindent for all $(t,s)\in U$. We show the above identity. We argue as for $1)$ using Proposition \ref{PRP.L2W_Log_Mean_NCDS_Fin_Hessian} and Lemma \ref{LEM.L2W_Log_Mean_NCDS_Fin_EL}. We likewise use symmetry of the real inner product. Thus

\begin{align}\label{EQ.THM.L2W_Log_Mean_NCDS_Fin_Hessian_7}
\Hess\hspace{-0.1525cm} \phantom{.}_{\eta}\Enttau\lc\mathfrak{F}_{\eta}(X)^{\flat}\rc{}=\Hess\hspace{-0.1525cm} \phantom{.}_{\eta}\Enttau\lc\mathfrak{F}_{\eta,X}^{\flat}\rc{}
\end{align}

\noindent is of required form, hence $2)$ holds.
\end{proof}


\subsection{Quantum noise evolution}\label{SSEC.L2W_Log_Mean_QNE}

We view quantum Laplacians as generators of quantum noise evolution in order to have non-spatiality of lower Ricci bounds and associated energy-information trade-offs. If $\EVI_{\lambda}$-gradient flow of quantum relative entropy exist, then our Corollary \ref{COR.L2W_EVI_Equivalence} shows it is heat flow. Its curves of maximal slope \cite{ART.Mur_Sav.2020.Classical_OT_EVI} determine slopes of maximal entropy production, i.e.~erasure of quantum information. A priori, it is nevertheless unclear how the $\EVI_{\lambda}$-gradient flow property selects noise diffusion terms, i.e.~generators of quantum noise evolution, without their selection being an isolated assumption unrelated to the underlying metric geometry. We require finer model assumptions for a selection process to justify viewing quantum Laplacians as above. To this end, we formulate a maximum entropy production principle as the latter may determine erasure of information \cite{ART.Dew.2003.MaxEnt_Information_I}\cite{ART.Dew.2005.MaxEnt_Information_II}\linebreak\cite{ART.Gri_Lin.2007.MaxEnt_Information_Refutation} motivated by fluctuation-dissipation principles \cite{ART.Alh_Bal_Rei.1986.StM_Information_Geometric}\cite{ART.AlRom_Rem.1996.StM_Information_Foundation}\cite{ART.Bon_Bru_Jur_Kui_Pet_Zup.2010.MaxEnt_LNonEq_Equivalence}\cite{ART.Mar_Sel.2006.MaxEnt_Review} in non-equilibrium classical \cite{BK.Bed_Kje.2008.StM_Non_Equilibrium}\cite{BK.Pri.1967.MaxEnt_Foundational} and quantum statistical mechanics \cite{BK.Ste_vLee.2013.Full_Quantum_StM}.\par


\pagebreak


Up to coarse graining, Lemma \ref{LEM.L2W_Log_Mean_NCDS} implies heat flow is gradient flow of quantum relative entropy. Theorem \ref{THM.L2W_Log_Mean_QNE} shows heat flow further satisfies a steepest entropy ascent property \cite{ART.Ber_Con_Mon.2015.MaxEnt_SEA} by considering the steepest descent property of gradient flows in smooth Riemannian manifolds \cite{BK.Lan.1995.Riemannian_Manifolds} and taking limits. Note Corollary \ref{COR.Rel_Ent_AF_Cstar_Trace} shows production of quantum entropy is erasure of quantum information. We seek conditions s.t.~steepest entropy ascent implies quantum noise evolution. If we are able to do so, then Theorem \ref{THM.L2W_Log_Mean_QNE} obtains slopes of maximal entropy production, i.e.~erasure of quantum information, for sufficiently regular subsets of all bounded normal states. Metric slopes as per Equation \ref{EQ.SSEC.L2W_EVI_Equivalence_1} generalise to larger sets of unbounded normal states. We restrict our maximum entropy production principle to selection of noise diffusion terms in the finite-dimensional setting and assume such selection is stable under scaling limits.\par
Accordingly, our maximum entropy production principle selects from candidates for noise diffusion terms in the finite-dimensional setting. Each candidate is determined by a quantum Fokker-Planck equation with vanishing drift term s.t.~the kernel of the given quantum Laplacian is the solution set for zero. Following Remark \ref{REM.Wstar_CP_Markovian_SG}, generators of induced semi\-groups as per Lemma \ref{LEM.Wstar_CP_Markovian_SG} satisfying a quantum Fokker-Planck equation with vanishing drift term are diffusion terms. These describe purely irreversible time-evolution of dissipative quantum systems weakly coupled to a heat bath \cite{BK.Bra.1987.OpAlg_Quantum_StM_I}\cite{BK.Bra.1987.OpAlg_Quantum_StM_II}\cite{BK.Dav.1976.Quantum_Markov_SG}\linebreak\cite{BK.Gar_Zol.2004.Quantum_Noise}\cite{BK.Ohy_Pet.1993.Rel_Ent}\cite{BK.Ste_vLee.2013.Full_Quantum_StM}. Following Landauer's principle \cite{ART.Lan.1961.Information_Physical_I}\cite{ART.Lan.1961.Information_Physical_II} and its extension to quantum information theory \cite{BK.Cam_Def.2019.Quantum_StM_Information}\cite{ART.DiVi_Loss.1998.Quantum_Information_Physical}, we expect they produce quantum entropy at each state. We show this is the case for candidates but with arbitrary energy scales. If we fix these, then we may formulate our selection rule. Note Corollary \ref{COR.Wstar_CP_Markovian_SG} shows the given quantum Laplacian has vanishing drift term, i.e.~is itself a candidate for noise diffusion terms.\par
We consider four model assumptions. The first three assume the finite-dimensional setting, and the latter is stability under scaling limits. We summarise the first three. First, we assume production of quantum entropy, i.e.~erasure of quantum information, is transport of quantum information along information-bearing degrees of freedom. This amounts to assuming the logarithmic mean setting and our above notion of candidate. Secondly, we select noise diffusion terms from all candidates for arbitrary energy scales by maximising production of quantum entropy under constraints on energy spent. Maximisation constraints are given by suitable evaluation of quantum Fisher information at each state. The latter links the information structure of quantum relative entropy to the energy structure of the given quasi-entropy, i.e.~the underlying metric geometry. Thirdly, we use fixed energy scales normalised relative to the given quantum Laplacian. We obtain normalisation from an equivalent but expected least dissipation of energy principle \cite{ART.Bon_Bru_Jur_Kui_Pet_Zup.2010.MaxEnt_LNonEq_Equivalence}. This ensures unique solutions and avoids implausible ones. \par
Under assumptions as above, our maximum entropy production principle then states self-adjoint local unbounded operators are generators of quantum noise evolution if they restrict to unique solutions in each case. Corollary \ref{COR.L2W_Log_Mean_QNE_MaxEnt} implies these are indeed negatives of quantum Laplacians. Following our discussion of the coarse graining process in Subsection \ref{SSEC.QOT_CG}, Theorem \ref{THM.L2W_Log_Mean_QNE} shows quantum Laplacians satisfy, up to sign, a quantum Fokker-Planck equation with vanishing drift term in scaling limit, i.e.~only noise diffusion term. Of course, the sign occurs since negatives of quantum Laplacians generate noncommutative heat semigroups as per Lemma \ref{LEM.Wstar_CP_Markovian_SG}.


\subsubsection*{The maximum entropy production principle}

We motivate our formulation in the finite-dimensional setting by fluctuation-dissipation principles \cite{ART.Alh_Bal_Rei.1986.StM_Information_Geometric}\cite{ART.AlRom_Rem.1996.StM_Information_Foundation}\cite{ART.Bon_Bru_Jur_Kui_Pet_Zup.2010.MaxEnt_LNonEq_Equivalence}\cite{ART.Mar_Sel.2006.MaxEnt_Review} in non-equilibrium classical \cite{BK.Bed_Kje.2008.StM_Non_Equilibrium}\cite{BK.Pri.1967.MaxEnt_Foundational} and quantum statistical mechanics \cite{BK.Ste_vLee.2013.Full_Quantum_StM}. The latter exist in form of both minimum and maximum entropy production principles depending on constraints imposed on the given time-evolution \cite{ART.Alh_Bal_Rei.1986.StM_Information_Geometric}\cite{ART.AlRom_Rem.1996.StM_Information_Foundation}\cite{ART.Bon_Bru_Jur_Kui_Pet_Zup.2010.MaxEnt_LNonEq_Equivalence}. The variational approach in \cite{ART.Rei_Zim.2015.MaxEnt_Production} derives $L^{2}$-Wasserstein gradient flows by considering infinitesimal constraints on energy spent. This extends Onsager's least dissipation of energy principle \cite{ART.Ons.1937.LEn_I}\cite{ART.Ons.1937.LEn_II}. In the setting of linear non-equilibrium thermodynamics \cite{ART.Alh_Bal_Rei.1986.StM_Information_Geometric}\cite{ART.AlRom_Rem.1996.StM_Information_Foundation}\cite{BK.Bed_Kje.2008.StM_Non_Equilibrium}\cite{BK.Pri.1967.MaxEnt_Foundational}, Onsager's least dissipation of energy principle is equivalent to a maximum entropy production principle \cite{ART.Bon_Bru_Jur_Kui_Pet_Zup.2010.MaxEnt_LNonEq_Equivalence}. There exist efforts to give a sensible description of the latter exclusively in terms of information theory \cite{ART.Dew.2003.MaxEnt_Information_I}\cite{ART.Dew.2005.MaxEnt_Information_II}. However, such a description is contested \cite{ART.Gri_Lin.2007.MaxEnt_Information_Refutation}. We still arrive at three formal conditions for a suitable maximum entropy production principle. First, it must consider exclusively infinitesimal data for its maximisation constraints on energy spent. Secondly, it must be equivalent to a least dissipation of energy principle for the given thermodynamics by choice of such constraints. Thirdly, these constraints must be described only in terms of quantum information theory \cite{BK.Nie_Chu.2000.Quantum_Computation_Information}. We show all three formal conditions are satisfied by our maximum entropy production principle.\par
We in fact derive it from an equivalent least dissipation of energy principle. As part of our discussion, we make explicit the first three model assumptions. Equation \ref{EQ.DFN.L2W_Log_Mean_QNE_MaxEnt_General_1} gives maximal production of quantum entropy for candidates of noise diffusion terms as per the first model assumptions. This lets us select noise diffusion terms for arbitrary energy scales as per the second model assumption. Unless we fix energy scales, Proposition \ref{PRP.L2W_Log_Mean_QNE_MaxEnergy} implies we do not have unique solutions. Lemma \ref{LEM.L2W_Log_Mean_QNE_MaxEnt_General}, which assumes Equation \ref{EQ.DFN.L2W_Log_Mean_QNE_MaxEnt_General_1}, leads us to normalised energy scales as per the third model assumption and thereby our least dissipation of energy principle s.t.~heat flow serves as fluctuated gradient flow. Equation \ref{EQ.DFN.L2W_Log_Mean_QNE_MaxEnt_3} gives the latter. Example \ref{BSP.L2W_Log_Mean_QNE_QG_Internal} shows our choice kills implausible solutions in the essential case of depolarising channels \cite{BK.Nie_Chu.2000.Quantum_Computation_Information}. Lemma \ref{LEM.L2W_Log_Mean_QNE_MaxEnt} shows Equation \ref{EQ.COR.L2W_Log_Mean_QNE_MaxEnt_1}, i.e.~Equation \ref{EQ.DFN.L2W_Log_Mean_QNE_MaxEnt_General_1} for normalised energy scales, is derived from Equation \ref{EQ.DFN.L2W_Log_Mean_QNE_MaxEnt_3} in Corollary \ref{COR.L2W_Log_Mean_QNE_MaxEnt}. Equation \ref{EQ.COR.L2W_Log_Mean_QNE_MaxEnt_1} selects noise diffusion terms in the finite-dimensional setting as per our maximum entropy production principle.\par
Let $(\phi,\bpsi,\gamma,\nabla)$ be noncommutative differential structure for tracial AF-$C^{*}$-algebras $(A,\tau)$ and $(B,\omega)$ in the finite-dimensional logarithmic mean setting. Proposition \ref{PRP.L2W_Log_Mean_QNE_GradFlow} shows heat flow is gradient flow of quantum relative entropy. Remark \ref{REM.L2W_Log_Mean_QNE_MaxEnergy} explains Proposition \ref{PRP.L2W_Log_Mean_QNE_MaxEnergy} gives maximisation constraints on energy spent for Equation \ref{EQ.DFN.L2W_Log_Mean_QNE_MaxEnt_General_1}.

\begin{ntn}
Let $\xi\in\SII(A)$ be a fixed state. Let $\grad\Enttau$ denote the gradient of $\Enttau$ restricted to $\vartheta(\xi)$. We write $\grad\hspace{-0.1525cm}\phantom{.}_{\mu}\Enttau$ upon evaluation at $\mu\in\vartheta(\xi)$.
\end{ntn}

\begin{prp}\label{PRP.L2W_Log_Mean_QNE_GradFlow}
Let $\xi\in\SII(A)$ be a fixed state. For all $\mu\in\vartheta(\xi)$, we have

\begin{align}\label{EQ.PRP.L2W_Log_Mean_QNE_GradFlow_1}
-\grad\hspace{-0.1525cm}\phantom{.}_{h_{t}(\mu)}\Enttau=-\lc\Delta\sharp h_{t}(\mu)\rc^{\flat}=\frac{d}{dt}h_{t}(\mu)
\end{align}

\noindent for all $t\geq 0$.
\end{prp}
\begin{proof}
Let $\mu\in\vartheta(\xi)$. Since $\frac{d}{dt}h_{t}(\mu)=-\lc\Delta\sharp h_{t}(\mu)\rc^{\flat}$ for all $t\geq 0$ by construction, we know Equation \ref{EQ.PRP.L2W_Log_Mean_QNE_GradFlow_1} follows if

\begin{align}\label{EQ.PRP.L2W_Log_Mean_QNE_GradFlow_2}
\grad\hspace{-0.1525cm}\phantom{.}_{\eta}\Enttau=\lc\Delta\sharp\eta\rc^{\flat}
\end{align}

\noindent for all $\eta\in\vartheta(\xi)$. We show Equation \ref{EQ.PRP.L2W_Log_Mean_QNE_GradFlow_2}. Let $\eta\in\vartheta(\xi)$ and $u\in T_{\eta}\vartheta(\xi)$. Let $\varepsilon>0$ and $\rho:\lb{}-\varepsilon,\varepsilon\rb\longrightarrow\vartheta(\xi)$ smooth s.t.~$\rho(0)=\eta$ and $\dot{\rho}(0)=u$. We directly verify having admissible path $\lc\rho,\Theta\lc\rho,\dot{\rho}\rc\rc\in\Adm^{\lb{}-\varepsilon,\varepsilon\rb{}}$ satisfying the conditions of $2)$ in Lemma \ref{LEM.L2W_Log_Mean_NCDS}. Using the latter and $\mathfrak{G}_{\eta}=\mathcal{M}_{\sharp\eta,\xi}\nabla$ ensured by Equation \ref{EQ.SSEC.QOT_AC_RM_3}, we calculate

\begin{align*}
\restr{0.925}{\frac{d}{dt}}{t=0}\Enttau\lc\rho(t)\rc{} & = \lgl\mathcal{D}_{\sharp\eta,\xi}\sharp\Theta\lc\rho,\dot{\rho}\rc(0),\nabla\sharp\eta\rgl_{\omega} \phantom{\vstretch{0.915}{\Bigg)}} \\
& = \lgl\mathcal{D}_{\sharp\eta,\xi}\mathcal{M}_{\sharp\eta,\xi}\nabla\mathfrak{F}_{\eta}^{-1}\lc\sharp u\rc{},\nabla\sharp\eta\rgl_{\omega} \phantom{\vstretch{0.915}{\Bigg)}} \\
& = \lgl\nabla\mathfrak{F}_{\eta}^{-1}\lc\sharp u\rc{},\nabla\sharp\eta\rgl_{\omega} \phantom{\vstretch{0.915}{\Bigg)}} \\
& = \lgl\mathfrak{F}_{\eta}^{-1}\lc\sharp u\rc{},\Delta\sharp\eta\rgl_{\tau} \phantom{\vstretch{0.915}{\Bigg)}} \\
& = g_{\eta}^{\xi}\big(u,\lc\Delta\sharp\eta\rc^{\flat}\big). \phantom{\vstretch{0.915}{\Bigg)}}
\end{align*}

\noindent The above calculation implies Equation \ref{EQ.PRP.L2W_Log_Mean_QNE_GradFlow_2} and therefore Equation \ref{EQ.PRP.L2W_Log_Mean_QNE_GradFlow_1}.
\end{proof}

\begin{dfn}\label{DFN.L2W_Log_Mean_QNE_MaxEnergy}
Let $\xi\in\SII(A)$ be a fixed state and $\mu\in\vartheta(\xi)$.

\begin{itemize}
\item[1)] We define $\mathfrak{H}_{\xi,\mu}:A_{\xi,h}\longrightarrow\mathbb{R}$ by setting

\begin{align}\label{EQ.DFN.L2W_Log_Mean_QNE_MaxEnergy_1}
\mathfrak{H}_{\xi,\mu}(x):=g_{\mu}^{\xi}\lc\big(\Delta x\big)^{\flat},\lc\Delta\sharp\mu\rc^{\flat}\rc{}
\end{align}

\begin{reapply}
\end{reapply}

\noindent for all $x\in A_{\xi,h}$.

\item[2)] For all $C\geq 0$, set $\mathfrak{S}_{\xi,\mu}(C):=\bigg\{\hspace{0.025cm} x\in A_{\xi,h}\ \big\vert\ g_{\mu}^{\xi}\lc\big(\Delta x\big)^{\flat},\big(\Delta x\big)^{\flat}\rc{}=C\hspace{0.025cm} \bigg\}$. 
\end{itemize}
\end{dfn}

\begin{rem}\label{REM.L2W_Log_Mean_QNE_MaxEnergy}
Let $\xi\in\SII(A)$ be a fixed state and $\mu\in\vartheta(\xi)$. Using definition of gradient flows \cite{BK.Lan.1995.Riemannian_Manifolds} and following $1)$ in Definition \ref{DFN.L2W_Log_Mean_QNE_MaxEnergy}, Proposition \ref{PRP.L2W_Log_Mean_QNE_GradFlow} shows

\begin{align}\label{EQ.REM.L2W_Log_Mean_QNE_MaxEnergy_1}
-\frac{d}{dt}\Enttau\lc{}h_{t}(\mu)\rc{}=g_{h_{t}(\mu)}^{\xi}\lc\lc\Delta\sharp h_{t}(\mu)\rc^{\flat},\lc\Delta\sharp h_{t}(\mu)\rc^{\flat}\rc{}=\mathfrak{H}_{\xi,h_{t}(\mu)}\lc\sharp h_{t}(\mu)\rc{}
\end{align}

\noindent for all $t\geq 0$. Equation \ref{EQ.REM.L2W_Log_Mean_QNE_MaxEnergy_1} for $t=0$ yields $-\restr{0.925}{\frac{d}{dt}}{t=0}\Enttau\lc{}h_{t}(\mu)\rc{}=\mathfrak{H}_{\xi,\mu}\lc\sharp\mu\rc$ at once. We use this throughout our discussion.
\end{rem}

\begin{prp}\label{PRP.L2W_Log_Mean_QNE_MaxEnergy}
Let $\xi\in\SII(A)$ be a fixed state and $\mu\in\vartheta(\xi)\setminus\lset\xi\rset{}$. Let $C>0$. For all $x\in\mathfrak{S}_{\xi,\mu}(C)$, we have $\mathfrak{H}_{\xi,\mu}(x)=\sup_{y\in\mathfrak{S}_{\xi,\mu}(C)}\mathfrak{H}_{\xi,\mu}(y)$ if and only if

\begin{align}\label{EQ.PRP.L2W_Log_Mean_QNE_MaxEnergy_1}
\Delta x=C^{\frac{1}{2}}\cdot\mathfrak{H}_{\xi,\mu}\lc\sharp\mu\rc^{-\frac{1}{2}}\cdot \Delta\sharp\mu.
\end{align}
\end{prp}
\begin{proof}
We consider Riemannian manifold $(\vartheta(\xi),g^{\xi})$ as per $1)$ in Proposition \ref{PRP.RM_II}. We know $T_{\mu}\vartheta(\xi)=I(\Delta_{\xi})^{\flat}$ by $2)$ in Proposition \ref{PRP.RM_Embedded_Submanifold_II}. Pull-back of $g_{\mu}^{\xi}$ along the flat operator yields real Hilbert space $\lc{}I(\Delta_{\xi}),g_{\mu}^{\xi}\rc$.\par
We have orthogonal decomposition $I(\Delta_{\xi})=\langle\Delta\sharp\mu\rangle_{\mathbb{R}}\oplus\langle\Delta\sharp\mu\rangle_{\mathbb{R}}^{\perp}$. For all $x\in\mathfrak{S}_{\xi,\mu}(C)$, get unique $C_{x}\in\mathbb{R}$ and $r_{x}\in\langle\Delta\sharp\mu\rangle_{\mathbb{R}}^{\perp}$ s.t.~

\begin{align}\label{EQ.PRP.L2W_Log_Mean_QNE_MaxEnergy_2}
\Delta x=C_{x}\cdot \Delta\sharp\mu+r_{x}.
\end{align}

\noindent Using $2)$ in Definition \ref{DFN.L2W_Log_Mean_QNE_MaxEnergy}, Equation \ref{EQ.PRP.L2W_Log_Mean_QNE_MaxEnergy_2} shows

\begin{align}\label{EQ.PRP.L2W_Log_Mean_QNE_MaxEnergy_3}
C=C_{x}^{2}\cdot \mathfrak{H}_{\xi,\mu}\lc\sharp\mu\rc{}+g_{\mu}^{\xi}(r_{x},r_{x})
\end{align}

\noindent for all $x\in\mathfrak{S}_{\xi,\mu}(C)$. Since moreover $r_{x}\in\langle\Delta\sharp\mu\rangle_{\mathbb{R}}^{\perp}$, Equation \ref{EQ.PRP.L2W_Log_Mean_QNE_MaxEnergy_2} further shows

\begin{align}\label{EQ.PRP.L2W_Log_Mean_QNE_MaxEnergy_4}
\mathfrak{H}_{\xi,\mu}(x)=C_{x}\cdot \mathfrak{H}_{\xi,\mu}\lc\sharp\mu\rc{}
\end{align}

\noindent in each case. In addition, note Corollary \ref{COR.QOT_Distance_AC_L2} states $\xi\in\vartheta(\xi)$ is the only fixed state in $\vartheta(\xi)$. Yet $\mu\neq\xi$. Thus Proposition \ref{PRP.L2W_Log_Mean_QNE_GradFlow} implies $\mathfrak{H}_{\xi,\mu}\lc\sharp\mu\rc{}=-\restr{0.925}{\frac{d}{dt}}{t=0}\Enttau\lc{}h_{t}(\mu)\rc{}>0$, hence we see Equation \ref{EQ.PRP.L2W_Log_Mean_QNE_MaxEnergy_3} shows

\begin{align}\label{EQ.PRP.L2W_Log_Mean_QNE_MaxEnergy_5}
\absv{1.15}{C_{x}}=\sqrt{C-g_{\mu}^{\xi}(r_{x},r_{x})}\cdot \mathfrak{H}_{\xi,\mu}\lc\sharp\mu\rc^{-\frac{1}{2}}
\end{align}

\noindent for all $x\in\mathfrak{S}_{\xi,\mu}(C)$ by rearranging terms accordingly. Let $x\in\mathfrak{S}_{\xi,\mu}(C)$. Equation \ref{EQ.PRP.L2W_Log_Mean_QNE_MaxEnergy_4} shows we have $\mathfrak{H}_{\xi,\mu}(x)=\sup_{y\in\mathfrak{S}_{\xi,\mu}(C)}\mathfrak{H}_{\xi,\mu}(y)$ if and only if $C_{x}=\sup_{y\in\mathfrak{S}_{\xi,\mu}(C)}C_{y}$ holds. Up to positive constant, note $\sharp\mu\in\mathfrak{S}_{\xi,\mu}(C)$. We assume $C_{x}\geq 0$ without loss of generality since we are concerned with the supremum. Equation \ref{EQ.PRP.L2W_Log_Mean_QNE_MaxEnergy_5} therefore implies we have $\absv{1.15}{C_{x}}=C_{x}=\sup_{y\in\mathfrak{S}_{\xi,\mu}(C)}C_{y}$ if and only if

\begin{align}\label{EQ.PRP.L2W_Log_Mean_QNE_MaxEnergy_6}
g_{\mu}^{\xi}(r_{x},r_{x})=0.
\end{align}

\noindent Equation \ref{EQ.PRP.L2W_Log_Mean_QNE_MaxEnergy_6} states $r_{x}=0$ by positive definiteness of Riemannian metrics. Using the latter, Equation \ref{EQ.PRP.L2W_Log_Mean_QNE_MaxEnergy_2} and Equation \ref{EQ.PRP.L2W_Log_Mean_QNE_MaxEnergy_5} show the claimed equivalence.
\end{proof}


\pagebreak


We give explicit formulation of the first and second model assumption. For this, we must describe candidates for noise diffusion terms in the finite-dimensional setting and use the extension \cite{BK.Cam_Def.2019.Quantum_StM_Information}\cite{ART.DiVi_Loss.1998.Quantum_Information_Physical} of Landauer's principle \cite{ART.Lan.1961.Information_Physical_I}\cite{ART.Lan.1961.Information_Physical_II} to justify strictly positive production of quantum entropy. Equation \ref{EQ.SSEC.L2W_Log_Mean_QNE_1} shows such candidates are diffusion terms. Using $\ker\Delta$ as the solution set for zero in each case, Equation \ref{EQ.SSEC.L2W_Log_Mean_QNE_5} gives upper bounds on production of quantum entropy under constraints on energy spent. The latter ensure finiteness. Definition \ref{DFN.L2W_Log_Mean_QNE_MaxEnt_General} gives maximal production of quantum entropy as per Equation \ref{EQ.DFN.L2W_Log_Mean_QNE_MaxEnt_General_1} for arbitrary energy scales by maximising Equation \ref{EQ.SSEC.L2W_Log_Mean_QNE_5}.\par
We describe our notion of candidate. We use completely Markovian semigroups on $L^{\infty}(A,\tau)=A$ as per Definition \ref{DFN.Wstar_CP_Markovian_SG}. Let $S\in\BII(A)_{h}$ s.t.~$S\neq 0$ and $\ker S=\ker\Delta$. Assume $S$ has completely Markovian induced semigroup $G^{S}:[0,\infty)\longrightarrow\BII(A)$ given by $G_{t}^{S}=e^{tS}$ for all $t\geq 0$ as per $1)$ in Definition \ref{DFN.Wstar_CP_Markovian_SG_II}. We extend to positivity-preserving semigroup $G^{S}:[0,\infty)\longrightarrow\BII(A^{*})$ s.t.~$G_{t}^{S}(\SII(A))\subset\SII(A)$ for all $t\geq 0$ by dualisation. Self-adjointness implies $S$ is a diffusion term as follows. Corollary \ref{COR.Wstar_CP_Markovian_SG}, which uses Lemma \ref{LEM.Wstar_CP_Markovian_SG} in the finite-dimensional setting, shows there exists Lindblad decomposition $\lc{}0,\varphi,C\rc$ of $S$ as per $2)$ in Definition \ref{DFN.Wstar_CP_Markovian_SG_II}. Following Remark \ref{REM.Wstar_CP_Markovian_SG}, we therefore have a quantum Fokker-Planck equation given by

\begin{align}\label{EQ.SSEC.L2W_Log_Mean_QNE_1}
S(x)=\frac{C}{2}\lc{}2\varphi(x)-\big\{\varphi(1_{A}),x\big\}\rc{}
\end{align}

\noindent for all $x\in A$. Equation \ref{EQ.REM.Wstar_CP_Markovian_SG_1} shows Equation \ref{EQ.SSEC.L2W_Log_Mean_QNE_1} has vanishing drift term. We say that $S$ is a candidate for noise diffusion terms. As we show below, this notion of candidate is part of the first model assumption and leads us to the second one.\par
The first model assumption states production of quantum entropy, i.e.~erasure of quantum information, is transport of quantum information along information-bearing degrees of freedom. This description requires choice of quasi-entropy and measure of quantum information. We use $\mathcal{I}^{\log}$ and $\Enttau$ in our formulation here. Remark \ref{REM.L2W_Log_Mean_QNE_MaxEnt_General} explains our choice of the logarithmic mean setting. For all fixed states $\xi\in\SII(A)$, we replace $\mathcal{I}^{\log}$ with $g^{\xi}$ on $\vartheta(\xi)$ as per Remark \ref{REM.RM}. In Subsection \ref{SSEC.QOT_CG}, we explain non-ergodicity restricts information-bearing degrees of freedom by the continuity equation. Thus $\ker S=\ker\Delta$ restricts, hence $G^{S}:[0,\infty)\times\SII(A)\longrightarrow\SII(A)$ induces finite-energy admissible paths as follows. For all fixed states $\xi\in\SII(A)$, note $T\vartheta(\xi)=\vartheta(\xi)\times I(\Delta_{\xi})^{\flat}$ by $2)$ in Proposition \ref{PRP.RM_Embedded_Submanifold_II} and $\im S\cap A_{\xi,h}=\im\Delta\cap A_{\xi,h}=I(\Delta_{\xi})$ since $\ker S=\ker\Delta$. Using the latter, Corollary \ref{COR.QOT_Distance_AC_L2} then implies $\im S\cap A_{\xi,h}=I(\Delta_{\xi})$ is equivalent to the following statement in the finite-dimensional setting. For all fixed states $\xi\in\SII(A)$, we have

\begin{align}\label{EQ.SSEC.L2W_Log_Mean_QNE_2}
G_{t}^{S}\lc\vartheta(\xi)\rc\subset\vartheta(\xi)
\end{align}

\noindent for all $t\geq 0$. We have $\Delta\vert_{\im\Delta}>0$ in $\BII\lc\im\Delta\rc$ by finite-dimensionality. Equation \ref{EQ.SSEC.L2W_Log_Mean_QNE_2} yields finite-energy admissible paths in relative interiors. The first model assumption is use of noncommutative differential structure and notion of candidate as above.\par


\newpage


The second model assumption states we select noise diffusion terms from all candidates for arbitrary energy scales by maximising production of quantum entropy under constraints on energy spent. This requires candidates produce quantum entropy at each state. Following Remark \ref{REM.Wstar_CP_Markovian_SG}, we view diffusion terms as infinitesimal applications of quantum channels \cite{ART.Bih_Lid_Wha.2001.Quantum_Markov_CPM_Reconstruction}\cite{ART.Cub_Eis_Wol.2012.Quantum_Markov_SG_Reconstruction} transmitting change of states of the given quantum system determined by irreversible interaction with its environment \cite{BK.Nie_Chu.2000.Quantum_Computation_Information}\cite{ART.Kra.1971.State_Changes}. The extension \cite{BK.Cam_Def.2019.Quantum_StM_Information}\cite{ART.DiVi_Loss.1998.Quantum_Information_Physical} of Landauer's principle \cite{ART.Lan.1961.Information_Physical_I}\cite{ART.Lan.1961.Information_Physical_II} gives strictly positive lower bounds on production of quantum entropy upon application of quantum channels due to minimal heat dissipation \cite{ART.Ara_Ber_Cil_Dil_Lut_Pet.2012.Information_Physical_Experimental_Verification}\cite{ART.Burz_Gau_Lui_Mae_vandZan.2018.Quantum_Information_Physical_Experimental_Verification}\cite{ART.Sag_Ueda.2008.Quantum_Information_Physical_From_2nd_Law_ThDy}. Under assumptions identical to those for general Lindblad master equations \lc{}cf.~Equation 5.2.29 in \cite{BK.Gar_Zol.2004.Quantum_Noise}\rc{}, Equation 3.8 in \cite{BK.Cam_Def.2019.Quantum_StM_Information} shows erasure of quantum information implies strictly positive production of quantum entropy.\par
We expect $G^{S}:[0,\infty)\times\SII(A)\longrightarrow\SII(A)$ produces quantum entropy at each state since $S$ is infinitesimal application of $\varphi$. For all fixed states $\xi\in\SII(A)$, Equation \ref{EQ.SSEC.L2W_Log_Mean_QNE_2} and differentiation at $t=0$ yield unique $x_{\mu}\in I(\Delta_{\xi})$ s.t.~

\begin{align}\label{EQ.SSEC.L2W_Log_Mean_QNE_3}
S\lc\sharp\mu\rc{}=-\Delta x_{\mu}
\end{align}

\noindent for all $\mu\in\vartheta(\xi)$. Following Example \ref{BSP.Rel_Ent_Cstar_Fin_I}, Corollary \ref{COR.Rel_Ent_AF_Cstar_Trace} shows quantum entropy is negative quantum relative entropy. We give production of quantum entropy, i.e.~erasure of quantum information, at each state. For all fixed states $\xi\in\vartheta(\xi)$, $2)$ in Lemma \ref{LEM.L2W_Log_Mean_NCDS} as in the proof of Proposition \ref{PRP.L2W_Log_Mean_QNE_GradFlow} and Equation \ref{EQ.SSEC.L2W_Log_Mean_QNE_3} let us calculate

\begin{align}\label{EQ.SSEC.L2W_Log_Mean_QNE_4}
-\restr{0.925}{\frac{d}{dt}}{t=0}\Enttau\big(G_{t}^{S}(\mu)\big)=\tau\lc\Delta x_{\mu}\log\sharp\mu\rc{}=g_{\mu}^{\xi}\lc\lc\Delta x_{\mu}\rc^{\flat},\lc\Delta\sharp\mu\rc^{\flat}\rc{}=\mathfrak{H}_{\xi,\mu}\lc{}x_{\mu}\rc{}
\end{align}

\noindent for all $\mu\in\vartheta(\xi)$. Set $C_{\xi,\mu}:=g_{\mu}^{\xi}\lc\lc\Delta x_{\mu}\rc^{\flat},\lc\Delta x_{\mu}\rc^{\flat}\rc$ in each case. Proposition \ref{PRP.L2W_Log_Mean_QNE_MaxEnergy} shows these are energy scales, varying in each tangent space and which determine strictly positive constants in Equation \ref{EQ.PRP.L2W_Log_Mean_QNE_MaxEnergy_1} for the following maximisation problem. For all fixed states $\xi\in\vartheta(\xi)$, we have $x_{\mu}\in\mathfrak{S}_{\xi,\mu}(C_{\xi,\mu})$ and Equation \ref{EQ.SSEC.L2W_Log_Mean_QNE_4} shows

\begin{align}\label{EQ.SSEC.L2W_Log_Mean_QNE_5}
-\restr{0.925}{\frac{d}{dt}}{t=0}\Enttau\big(G_{t}^{S}(\mu)\big)\leq\sup_{y\in\mathfrak{S}_{\xi,\mu}(C_{\xi,\mu})}\hspace{0.025cm} \mathfrak{H}_{\xi,\mu}(y)
\end{align}

\noindent for all $\mu\in\vartheta(\xi)$. Maximising Equation \ref{EQ.SSEC.L2W_Log_Mean_QNE_5} gives rise to Definition \ref{DFN.L2W_Log_Mean_QNE_MaxEnt_General}, in particular to Equation \ref{EQ.DFN.L2W_Log_Mean_QNE_MaxEnt_General_1}. The second model assumption is selection of noise diffusion terms for arbitrary energy scales from all candidates through maximal production of quantum entropy as per Equation \ref{EQ.DFN.L2W_Log_Mean_QNE_MaxEnt_General_1} by maximising Equation \ref{EQ.SSEC.L2W_Log_Mean_QNE_5}.

\begin{dfn}\label{DFN.L2W_Log_Mean_QNE_MaxEnt_General}
Let $S\in\BII(A)_{h}$ s.t.~$S\neq 0$ and $\ker S=\ker\Delta$. Assume $S$ has completely Markovian induced semigroup $G^{S}:[0,\infty)\longrightarrow\BII(A)$. We say that $S$ produces maximal quantum entropy for $\nabla$ if for all fixed states $\xi\in\SII(A)$ and $\mu\in\vartheta(\xi)$, we have $C\geq 0$ s.t.~

\begin{align}\label{EQ.DFN.L2W_Log_Mean_QNE_MaxEnt_General_1}
-\restr{0.925}{\frac{d}{dt}}{t=0}\Enttau\big(G_{t}^{S}(\mu)\big)=\sup_{y\in\mathfrak{S}_{\xi,\mu}(C)}\hspace{0.025cm} \mathfrak{H}_{\xi,\mu}(y).
\end{align}
\end{dfn}


\pagebreak


\begin{rem}\label{REM.L2W_Log_Mean_QNE_MaxEnt_General}
Note $1)$ in Proposition \ref{PRP.Wstar_Derivation_QG_HSG_II} ensures $-\Delta$ is a candidate. Following Remark \ref{REM.L2W_Log_Mean_QNE_MaxEnergy} for $t=0$, Proposition \ref{PRP.L2W_Log_Mean_QNE_GradFlow} and Proposition \ref{PRP.L2W_Log_Mean_QNE_MaxEnergy} show $-\Delta$ produces maximal quantum entropy using energy scale

\begin{align}\label{EQ.L2W_Log_Mean_QNE_MaxEnt_General_1}
C_{\xi,\mu}=\mathfrak{H}_{\xi,\mu}\lc\sharp\mu\rc^{\frac{1}{2}}
\end{align}

\noindent for all fixed states $\xi\in\SII(A)$ and $\mu\in\vartheta(\xi)$. If we do use both the first and second model assumptions, then $-\Delta$ is a noise diffusion term for energy scale as per Equation \ref{EQ.L2W_Log_Mean_QNE_MaxEnt_General_1}. We expect this but require Proposition \ref{PRP.L2W_Log_Mean_QNE_GradFlow} and Proposition \ref{PRP.L2W_Log_Mean_QNE_MaxEnergy}. Moreover, the two propositions are necessary to derive Equation \ref{EQ.SSEC.L2W_Log_Mean_QNE_5} and therefore Equation \ref{EQ.DFN.L2W_Log_Mean_QNE_MaxEnt_General_1}. This in turn requires us to assume the logarithmic mean setting.
\end{rem}

Example \ref{BSP.L2W_Log_Mean_QNE_QG_Internal} shows selection of noise diffusion terms as per the second model assumption must discern multiples of $-\Delta$. Unless we fix energy scales, Proposition \ref{PRP.L2W_Log_Mean_QNE_MaxEnergy} shows we do not. Lemma \ref{LEM.L2W_Log_Mean_QNE_MaxEnt_General} shows candidates producing maximal quantum entropy are determined by energy maps varying $-\Delta$. This leads us to normalised energy scales as per the third model assumption and thereby our least dissipation of energy principle s.t.~heat flow serves as fluctuated gradient flow.

\begin{lem}\label{LEM.L2W_Log_Mean_QNE_MaxEnt_General}
Let $S\in\BII(A)_{h}$ s.t.~$S\neq 0$ and $\ker S=\ker\Delta$. Assume $S$ has completely Markovian induced semigroup $G^{S}:[0,\infty)\longrightarrow\BII(A)$. If $S$ produces maximal quantum entropy for $\nabla$, then we know there exist two unique maps $E_{S}:\partial\SII(A)\times [0,\infty)\longrightarrow [0,\infty)$ and $\lambda_{S}:\partial\SII(A)\times (0,\infty)\longrightarrow (0,\infty)$ satisfying the following.

\begin{itemize}
\item[1)] The map $\restr{0.925}{E_{S}}{\partial\SII(A)}:\partial\SII(A)\times\lset{}0\rset\longrightarrow (0,\infty)$ is norm continuous.

\item[2)] For all $\mu\in\partial\SII(A)$, the map $E_{S}(\mu,\blank):(0,\infty)\longrightarrow (0,\infty)$ is continuously differentiable and the map $\lambda_{S}:(\mu,\blank):(0,\infty)\longrightarrow (0,\infty)$ is continuous.

\item[3)] For all $\mu\in\partial\SII(A)$, we have

\begin{itemize}
\item[3.1)] $E_{S}(\mu,t)=\big\| S\big(G_{t}^{S}\lc\sharp\mu\rc\big)\big\|_{\tau}\big\|\Delta G_{t}^{S}\lc\sharp\mu\rc\big\|_{\tau}^{-1}$ for all $t>0$, \phantom{\big)}

\item[3.2)] $\restr{0.925}{E_{S}}{\partial\SII(A)}(\mu)=E_{S}(\mu,0)=\lim_{t\downarrow 0}\big\| S\big(G_{t}^{S}\lc\sharp\mu\rc\big)\big\|_{\tau}\big\|\Delta G_{t}^{S}\lc\sharp\mu\rc\big\|_{\tau}^{-1}$. \phantom{\big)}
\end{itemize}

\begin{reapply}
\end{reapply}

\item[4)] For all $\mu\in\partial\SII(A)$, we have

\begin{align}\label{EQ.LEM.L2W_Log_Mean_QNE_MaxEnt_General_1}
S\big(G_{t}^{S}\lc\sharp\mu\rc\big)=-E_{S}(\mu,t)\cdot \Delta G_{t}^{S}\lc\sharp\mu\rc{}
\end{align}

\begin{reapply}
\end{reapply}

\noindent for all $t\geq 0$.

\item[5)] For all $\mu\in\partial\SII(A)$, we have

\begin{align}\label{EQ.LEM.L2W_Log_Mean_QNE_MaxEnt_General_2}
\frac{d}{dt}E_{S}(\mu,t)=\lambda_{S}(\mu,t)\cdot E_{S}(\mu,t)
\end{align}

\begin{reapply}
\end{reapply}

\noindent for all $t>0$.
\end{itemize}
\end{lem}

\begin{proof}
Let $\xi\in\SII(A)$ be a fixed state and $\mu\in\vartheta(\xi)\setminus\lset\xi\rset{}$. Note $\Delta x_{\mu}\neq 0$ and $\mathfrak{H}_{\xi,\mu}\lc\sharp\mu\rc{}>0$ since we have $\sharp\mu\notin\ker S=\ker\Delta$. Equation \ref{EQ.SSEC.L2W_Log_Mean_QNE_5} shows 

\begin{align}\label{EQ.LEM.L2W_Log_Mean_QNE_MaxEnt_General_3}
C_{\xi,\mu}:=g_{\mu}^{\xi}\lc\lc\Delta x_{\mu}\rc^{\flat},\lc\Delta x_{\mu}\rc^{\flat}\rc{}>0   
\end{align}

\noindent is the unique constant in Equation \ref{EQ.DFN.L2W_Log_Mean_QNE_MaxEnt_General_1}. Then Equation \ref{EQ.SSEC.L2W_Log_Mean_QNE_4} and Equation \ref{EQ.DFN.L2W_Log_Mean_QNE_MaxEnt_General_1} let us calculate

\begin{align}\label{EQ.LEM.L2W_Log_Mean_QNE_MaxEnt_General_4}
\mathfrak{H}_{\xi,\mu}\lc{}x_{\mu}\rc{}=-\restr{0.925}{\frac{d}{dt}}{t=0}\Enttau\big(G_{t}^{S}(\mu)\big)=\sup_{y\in\mathfrak{S}_{\xi,\mu}(C_{\xi,\mu})}\hspace{0.025cm} \mathfrak{H}_{\xi,\mu}(y).
\end{align}

\noindent Using Proposition \ref{PRP.L2W_Log_Mean_QNE_MaxEnergy} for the second identity in Equation \ref{EQ.LEM.L2W_Log_Mean_QNE_MaxEnt_General_5} below, Equation \ref{EQ.SSEC.L2W_Log_Mean_QNE_3} and Equation \ref{EQ.LEM.L2W_Log_Mean_QNE_MaxEnt_General_4} show 

\begin{align}\label{EQ.LEM.L2W_Log_Mean_QNE_MaxEnt_General_5}
S\lc\sharp\mu\rc{}=-\Delta x_{\mu}=-C_{\xi,\mu}^{\frac{1}{2}}\cdot \mathfrak{H}_{\xi,\mu}\lc\sharp\mu\rc^{-\frac{1}{2}}\cdot \Delta\sharp\mu.
\end{align}

Equation \ref{EQ.LEM.L2W_Log_Mean_QNE_MaxEnt_General_7} uses constants on the right-hand side of Equation \ref{EQ.LEM.L2W_Log_Mean_QNE_MaxEnt_General_5} in order to define the claimed energy map on $\partial\SII(A)\times (0,\infty)$. Equation \ref{EQ.LEM.L2W_Log_Mean_QNE_MaxEnt_General_9} extends to $t=0$ in the second variable. For all fixed states $\xi\in\SII(A)$, $\mu\in\mathcal{C}_{A}(\xi)\setminus\lset\xi\rset{}$ and $t\geq 0$, we calculate $G_{t}^{S}(\mu)\neq\xi$ and therefore

\begin{align}\label{EQ.LEM.L2W_Log_Mean_QNE_MaxEnt_General_6}
\Delta G_{t}^{S}\lc\sharp\mu\rc\neq 0    
\end{align}

\noindent on an orthonormal eigenbasis of $S$. We define $E_{S}:\partial\SII(A)\times (0,\infty)\longrightarrow (0,\infty)$ by setting

\begin{align}\label{EQ.LEM.L2W_Log_Mean_QNE_MaxEnt_General_7}
E_{S}(\mu,t):=C_{\xi,\mu}^{\frac{1}{2}}\cdot \mathfrak{H}_{\xi,\mu}\big(G_{t}^{S}\lc\sharp\mu\rc\big)^{-\frac{1}{2}}
\end{align}

\noindent for all $\mu\in\partial\SII(A)$ and $t>0$. Equation \ref{EQ.LEM.L2W_Log_Mean_QNE_MaxEnt_General_6} and Equation \ref{EQ.LEM.L2W_Log_Mean_QNE_MaxEnt_General_7} show

\begin{align}\label{EQ.LEM.L2W_Log_Mean_QNE_MaxEnt_General_8}
E_{S}(\mu,t)=\dblv{}S\big(G_{t}^{S}\lc\sharp\mu\rc\big)\dblv_{\tau}\cdot \dblv{}\Delta G_{t}^{S}\lc\sharp\mu\rc\dblv_{\tau}^{-1}
\end{align}

\noindent in each case by taking Hilbert space norms and then the inverses in Equation \ref{EQ.LEM.L2W_Log_Mean_QNE_MaxEnt_General_5}. Using boundedness of $S$ and $\Delta$, Equation \ref{EQ.LEM.L2W_Log_Mean_QNE_MaxEnt_General_6} and Equation \ref{EQ.LEM.L2W_Log_Mean_QNE_MaxEnt_General_8} show we extend to $E_{S}:\partial\SII(A)\times [0,\infty)\longrightarrow (0,\infty)$ by setting

\begin{align}\label{EQ.LEM.L2W_Log_Mean_QNE_MaxEnt_General_9}
E_{S}(\mu,0):=\lim_{t\downarrow 0}\hspace{0.025cm}\dblv{}S\big(G_{t}^{S}\lc\sharp\mu\rc\big)\dblv_{\tau}\dblv{}\Delta G_{t}^{S}\lc\sharp\mu\rc\dblv_{\tau}^{-1}
\end{align}

\noindent for all $\mu\in\partial\SII(A)$. With the exception of $5)$, Equation \ref{EQ.LEM.L2W_Log_Mean_QNE_MaxEnt_General_8} and Equation \ref{EQ.LEM.L2W_Log_Mean_QNE_MaxEnt_General_9} show all claims involving $E_{S}$ here.\par
We show $5)$. Equation \ref{EQ.LEM.L2W_Log_Mean_QNE_MaxEnt_General_13} uses Equation \ref{EQ.LEM.L2W_Log_Mean_QNE_MaxEnt_General_12} in order to define the claimed map. Let $\mu\in\relint\SII(A)$. Using boundedness of $S$ and $\Delta$, as well as norm differentiability as per $1)$ and the Leibniz rule, Equation \ref{EQ.LEM.L2W_Log_Mean_QNE_MaxEnt_General_5} and Equation \ref{EQ.LEM.L2W_Log_Mean_QNE_MaxEnt_General_7} let us calculate

\begin{align*}
\restr{0.925}{\frac{d}{dr}}{r=0}S\big(G_{r}^{S}\lc\sharp\mu\rc\big) & = \restr{0.925}{\frac{d}{dr}}{r=0}G_{r}^{S}\lc{}S\lc\sharp\mu\rc\rc \phantom{\bigg)} \\
& = -E_{S}(\mu,0)\cdot S\lc\Delta\sharp\mu\rc \phantom{\bigg)} \\
& = -E_{S}(\mu,0)\cdot \lc\Delta S\lc\sharp\mu\rc{}+\big[S,\Delta\big]\lc\sharp\mu\rc\rc \phantom{\bigg)} \\
& = -E_{S}(\mu,0)\cdot \lc{}-E_{S}(\mu,0)\Delta^{2}\cdot \sharp\mu+\big[S,\Delta\big]\lc\sharp\mu\rc\rc \phantom{\bigg)} \\
& = E_{S}(\mu,0)^{2}\cdot \Delta^{2}\sharp\mu-E_{S}(\mu,0)\cdot \big[S,\Delta\big]\lc\sharp\mu\rc{}, \phantom{\bigg)} \\
& \\
\restr{0.925}{\frac{d}{dr}}{r=0}S\big(G_{r}^{S}\lc\sharp\mu\rc\big) & = -\restr{0.925}{\frac{d}{dr}}{r=0}E_{S}(\mu,r)\cdot \Delta G_{r}^{S}\lc\sharp\mu\rc \phantom{\Bigg)} \\
& = -\restr{0.925}{\frac{d}{dr}}{r=0}E_{S}(\mu,r)\cdot \Delta\sharp\mu-E_{S}(\mu,0)\cdot \Delta\restr{0.925}{\frac{d}{dr}}{r=0}G_{r}^{S}\lc\sharp\mu\rc \phantom{\Bigg)} \\
& = -\restr{0.925}{\frac{d}{dr}}{r=0}E_{S}(\mu,r)\cdot \Delta\sharp\mu-E_{S}(\mu,0)\cdot \Delta S\lc\sharp\mu\rc \phantom{\Bigg)} \\
& = -\restr{0.925}{\frac{d}{dr}}{r=0}E_{S}(\mu,r)\cdot \Delta\sharp\mu+E_{S}(\mu,0)^{2}\Delta^{2}\cdot \sharp\mu. \phantom{\Bigg)}
\end{align*}

\noindent We combine the two calculations above. We obtain

\begin{align}\label{EQ.LEM.L2W_Log_Mean_QNE_MaxEnt_General_10}
E_{S}(\mu,0)\cdot \big[S,\Delta\big]\lc\sharp\mu\rc\phantom{\bigg)}=\restr{0.925}{\frac{d}{dr}}{r=0}E_{S}(\mu,r)\cdot \Delta\sharp\mu.
\end{align}

\noindent For all $t\geq 0$, Equation \ref{EQ.LEM.L2W_Log_Mean_QNE_MaxEnt_General_10} shows $E_{S}(\mu,t)=E_{S}\lc{}G_{t}^{S}(\mu),0\rc$. Using the latter together with the semigroup property of $G^{S}:[0,\infty)\longrightarrow\BII(A)$, Equation \ref{EQ.LEM.L2W_Log_Mean_QNE_MaxEnt_General_10} generalises to

\begin{align}\label{EQ.LEM.L2W_Log_Mean_QNE_MaxEnt_General_11}
E_{S}(\mu,t)\cdot \big[S,\Delta\big]\big(G_{t}^{S}\lc\sharp\mu\rc\big)\phantom{\bigg)}=\restr{0.925}{\frac{d}{dr}}{r=t}E_{S}(\mu,r)\cdot \Delta G_{t}^{S}\lc\sharp\mu\rc{}
\end{align}

\noindent in each case. Equation \ref{EQ.LEM.L2W_Log_Mean_QNE_MaxEnt_General_8} therefore shows we have

\begin{align}\label{EQ.LEM.L2W_Log_Mean_QNE_MaxEnt_General_12}
\big[S,\Delta\big]\big(G_{t}^{S}\lc\sharp\mu\rc\big)\phantom{\bigg)}=E_{S}(\mu,t)^{-1}\cdot \restr{0.925}{\frac{d}{dr}}{r=t}E_{S}(\mu,r)\cdot \Delta G_{t}^{S}\lc\sharp\mu\rc{}
\end{align}

\noindent for all $t>0$ by taking the inverses in Equation \ref{EQ.LEM.L2W_Log_Mean_QNE_MaxEnt_General_11}.\par


\pagebreak


Equation \ref{EQ.LEM.L2W_Log_Mean_QNE_MaxEnt_General_12} shows we define $\lambda_{S}:\partial\SII(A)\times (0,\infty)\longrightarrow (0,\infty)$ by setting

\begin{align}\label{EQ.LEM.L2W_Log_Mean_QNE_MaxEnt_General_13}
\lambda_{S}(\mu,t):=E_{S}(\mu,t)^{-1}\cdot \restr{0.925}{\frac{d}{dr}}{r=t}E_{S}(\mu,r)
\end{align}

\noindent for all $\mu\in\partial\SII(A)$ and $t>0$. Equation \ref{EQ.LEM.L2W_Log_Mean_QNE_MaxEnt_General_6} and Equation \ref{EQ.LEM.L2W_Log_Mean_QNE_MaxEnt_General_12} show continuity in the second variable. Altogether, get $1)$ to $4)$. Equation \ref{EQ.LEM.L2W_Log_Mean_QNE_MaxEnt_General_12} and Equation \ref{EQ.LEM.L2W_Log_Mean_QNE_MaxEnt_General_13} yield

\begin{align}\label{EQ.LEM.L2W_Log_Mean_QNE_MaxEnt_General_14}
\lc\lambda_{S}(\mu,t)\cdot E_{S}(\mu,t)-\restr{0.925}{\frac{d}{dr}}{r=t}E_{S}(\mu,r)\rc\cdot \Delta G_{t}^{S}\lc\sharp\mu\rc{}=0
\end{align}

\noindent in each case. Equation \ref{EQ.LEM.L2W_Log_Mean_QNE_MaxEnt_General_6} and Equation \ref{EQ.LEM.L2W_Log_Mean_QNE_MaxEnt_General_14} imply Equation \ref{EQ.LEM.L2W_Log_Mean_QNE_MaxEnt_General_2}. Get $5)$.
\end{proof}

We give explicit formulation of the third model assumption. For this, we use our least dissipation of energy principle. Lemma \ref{LEM.L2W_Log_Mean_QNE_MaxEnt_General} lets us construct infinitesimal energy dis\-sipation maps as per Equation \ref{EQ.SSEC.L2W_Log_Mean_QNE_10}, resp.~its reformulation as Equation \ref{EQ.DFN.L2W_Log_Mean_QNE_MaxEnt_2}. We use Equation \ref{EQ.DFN.L2W_Log_Mean_QNE_MaxEnt_2} as measure of energy dissipation when evolving induced semigroups to heat flow through dissipating fluctuations of its integral curves. Definition \ref{DFN.L2W_Log_Mean_QNE_MaxEnt} gives least dissipation of energy as per Equation \ref{EQ.DFN.L2W_Log_Mean_QNE_MaxEnt_3} s.t.~heat flow serves as fluctuated gradient flow by minimising Equation \ref{EQ.DFN.L2W_Log_Mean_QNE_MaxEnt_2}. Accordingly, $3)$ in Definition \ref{DFN.L2W_Log_Mean_QNE_MaxEnt} gives candidates for noise diffusion terms with normal energy scale, i.e.~candidates satisfying Equation \ref{EQ.SSEC.L2W_Log_Mean_QNE_11}. The latter equation normalises energy scales relative to $-\Delta$.\par
We derive Equation \ref{EQ.SSEC.L2W_Log_Mean_QNE_10} and Equation \ref{EQ.SSEC.L2W_Log_Mean_QNE_11}. Let $S\in\BII(A)_{h}$ as per Lemma \ref{LEM.L2W_Log_Mean_QNE_MaxEnt_General}. For all $\mu\in\partial\SII(A)$, Equation \ref{EQ.LEM.L2W_Log_Mean_QNE_MaxEnt_General_2} readily shows $E_{S}(\mu,\blank):[0,\infty)\longrightarrow [0,\infty)$ satisfies a homogeneous linear differential equation with $3.2)$ in Lemma \ref{LEM.L2W_Log_Mean_QNE_MaxEnt_General} as its initial value at $t=0$ using standard arguments for extension \cite{BK.Eng_Nag.2000.Semigroups}\cite{BK.Koe.1993.Analysis_II}\cite{BK.Koe.2004.Analysis_I}. We therefore obtain

\begin{align}\label{EQ.SSEC.L2W_Log_Mean_QNE_6}
E_{S}(\mu,t)=\exp\lc\int_{0}^{t}\lambda_{S}(\mu,r)dr\rc\cdot E_{S}(\mu,0)
\end{align}

\noindent for all $\mu\in\partial\SII(A)$ and $t\geq 0$. Lemma \ref{LEM.L2W_Log_Mean_QNE_MaxEnt_General} ensures

\begin{align}\label{EQ.SSEC.L2W_Log_Mean_QNE_7}
\exp\lc\int_{0}^{t}\lambda_{S}(\mu,r)dr\rc\geq 1
\end{align}

\noindent since $\int_{0}^{t}\lambda_{S}(\mu,r)dr\geq 0$ in each case. Note $2.2)$ in Theorem \ref{THM.Wstar_Derivation_QG_HSG_Regularity} and Corollary \ref{COR.L2W_Log_Mean_NCDS} show $h:[0,\infty)\times\partial\SII(A)\longrightarrow\SII(A)$ is a norm continuous injective map s.t.~fixed states are the only elements not in its image. Moreover, Corollary \ref{COR.QOT_Distance_AC_L2} ensures all fixed states are limits in time of initial states in $\partial\SII(A)$. Thus $\partial\SII(A)\times [0,\infty)$ is a complete product space description of heat flow, hence we adopt it to measure infinitesimal energy dissipation when evolving the Hamiltonian of a given quantum system with initial state $\mu\in\partial\SII(A)$ from $S$ to $-\Delta$ at time $t\geq 0$. We formally view such evolutions as arising from dissipating small time-varying out-of-equilibrium perturbations of $-\Delta$, i.e.~fluctuations of integral curves $t\mapsto h_{t}(\mu)$ describing evolution of temperature over time \cite{BK.Bed_Kje.2008.StM_Non_Equilibrium}\cite{BK.Pri.1967.MaxEnt_Foundational}\cite{BK.Ste_vLee.2013.Full_Quantum_StM}.\par
We construct a suitable pointwise direct sum norm. Equation \ref{EQ.SSEC.L2W_Log_Mean_QNE_6} itself leads us to consider an energy gradient of $S$ given at $(\mu,t)\in\partial\SII(A)\times [0,\infty)$ by

\begin{align}\label{EQ.SSEC.L2W_Log_Mean_QNE_8}
\exp\lc\int_{0}^{t}\lambda_{S}(\mu,r)dr\rc\cdot \lc\inf_{\mu\in\partial\SII(A)}\hspace{0.025cm} E_{S}(\mu,0)\ -\sup_{\mu\in\partial\SII(A)}\hspace{0.025cm} E_{S}(\mu,0)\rc{}
\end{align}

\noindent for all $\mu\in\partial\SII(A)$ and $t\geq 0$. Equation \ref{EQ.SSEC.L2W_Log_Mean_QNE_8} is composed into two factors. The right-hand factor is the energy gradient of $S$ at $t=0$, or initial energy gradient of $S$ as per $2.1)$ in Definition \ref{DFN.L2W_Log_Mean_QNE_MaxEnt}. The left-hand factor is an exponential fluctuation term. If the initial energy gradient of $S$ is zero, then Equation \ref{EQ.SSEC.L2W_Log_Mean_QNE_6} and Equation \ref{EQ.SSEC.L2W_Log_Mean_QNE_7} imply variation of $G^{S}:[0,\infty)\times\SII(A)\longrightarrow\SII(A)$ away from heat flow is determined by the exponential fluctuation term up to homogeneous initial energy

\begin{align}\label{EQ.SSEC.L2W_Log_Mean_QNE_9}
e_{S}:= \inf_{\mu\in\partial\SII(A)}\hspace{0.025cm} E_{S}(\mu,0)\ =\sup_{\mu\in\partial\SII(A)}\hspace{0.025cm} E_{S}(\mu,0)
\end{align}

\noindent relative to $-\Delta$. If the exponential fluctuation term equals one, then Equation \ref{EQ.SSEC.L2W_Log_Mean_QNE_6} shows such variation is instead determined by initial states. We consequently measure infinitesimal energy dissipation when evolving $S$ to $-\Delta$ at initial state $\mu\in\partial\SII(A)$ and time $t\geq 0$ using the pointwise direct sum norm

\begin{align}\label{EQ.SSEC.L2W_Log_Mean_QNE_10}
\sqrt{\bbbabsv{1}{\exp\lc\int_{0}^{t}\lambda_{S}(\mu,r)dr\rc{}-1}^{2}+\bbbabsv{1}{\inf_{\mu\in\partial\SII(A)}\hspace{0.025cm} E_{S}(\mu,0)\ -\sup_{\mu\in\partial\SII(A)}\hspace{0.025cm} E_{S}(\mu,0)}^{2}}.
\end{align}

Equation \ref{EQ.SSEC.L2W_Log_Mean_QNE_7} shows Equation \ref{EQ.SSEC.L2W_Log_Mean_QNE_10} has zero as its minimum. Unless we restrict values of homogeneous initial energy as per Equation \ref{EQ.SSEC.L2W_Log_Mean_QNE_9}, Equation \ref{EQ.SSEC.L2W_Log_Mean_QNE_6} implies minimisers are given by $-C\Delta$ for all $C>0$. We expect this from Proposition \ref{PRP.L2W_Log_Mean_QNE_MaxEnergy} but instead due to energy scales varying away from Equation \ref{EQ.L2W_Log_Mean_QNE_MaxEnt_General_1}, i.e.~the energy scale of $-\Delta$, rather than from $e_{-\Delta}=1$. Note $3)$ in Lemma \ref{LEM.L2W_Log_Mean_QNE_MaxEnt_General} shows the latter. We therefore normalise energy scales relative to $-\Delta$ by letting

\begin{align}\label{EQ.SSEC.L2W_Log_Mean_QNE_11}
\inf_{\mu\in\partial\SII(A)}\hspace{0.025cm} E_{S}(\mu,0)\leq 1\leq\sup_{\mu\in\partial\SII(A)}\hspace{0.025cm} E_{S}(\mu,0).
\end{align}

\noindent Equation \ref{EQ.SSEC.L2W_Log_Mean_QNE_11} shows $-\Delta$ is the unique minimiser of Equation \ref{EQ.SSEC.L2W_Log_Mean_QNE_10}, i.e.~we have zero variation if and only if $E_{S}(\mu,t)=E_{-\Delta}(\mu,t)=1$ for all $\mu\in\partial\SII(A)$ and $t\geq 0$. The third model assumption is use of fixed energy scales normalised as per Equation \ref{EQ.SSEC.L2W_Log_Mean_QNE_11}.\par


\pagebreak


\begin{dfn}\label{DFN.L2W_Log_Mean_QNE_MaxEnt}
Let $S\in\BII(A)_{h}$ s.t.~$S\neq 0$ and $\ker S=\ker\Delta$. Assume $S$ has completely Markovian induced semigroup $G^{S}:[0,\infty)\longrightarrow\BII(A)$ and produces maximal quantum entropy for $\nabla$.

\begin{itemize}
\item[1)]  We call $E_{S}:\partial\SII(A)\times [0,\infty)\longrightarrow (0,\infty)$ the energy map of $S$. We further say that $\lambda_{S}:\partial\SII(A)\times (0,\infty)\longrightarrow (0,\infty)$ is its fluctuation.

\item[2)] Set $E_{S}^{\min}:=\inf_{\mu\in\partial\SII(A)}E_{S}(\mu,0)$ and $E_{S}^{\max}:=\sup_{\mu\in\partial\SII(A)}E_{S}(\mu,0)$.

\begin{itemize}
\item[2.1)] We define the initial energy gradient $\grad_{S}:=E_{S}^{\min}-E_{S}^{\max}$ of $S$. We define its variance $\var_{S}:\partial\SII(A)\times [0,\infty)\longrightarrow [0,\infty)$ by setting

\begin{align}\label{EQ.DFN.L2W_Log_Mean_QNE_MaxEnt_1}
\textrm{var}_{S}(\mu,t):=\exp\lc\int_{0}^{t}\lambda_{S}(\mu,r)dr\rc{}-1
\end{align}

\begin{reapply}
\end{reapply}

\noindent for all $\mu\in\partial\SII(A)$ and $t\geq 0$.

\item[2.2)] We define infinitesimal energy dissipation map $E_{S}^{\textrm{dis}}:\partial\SII(A)\times [0,\infty)\longrightarrow [0,\infty)$ of $S$ by setting

\begin{align}\label{EQ.DFN.L2W_Log_Mean_QNE_MaxEnt_2}
E_{S}^{\textrm{dis}}(\mu,t):=\sqrt{\babsv{1}{\textrm{var}_{S}(\mu,t)}^{2}+\babsv{1}{\textrm{grad}_{S}}^{2}}
\end{align}

\begin{reapply}
\end{reapply}

\noindent for all $\mu\in\partial\SII(A)$ and $t\geq 0$.
\end{itemize}

\begin{reapply}
\end{reapply}

\item[3)] We say that $S$ is a candidate for generating quantum noise evolution for $\nabla$ with normal energy scale if $E_{S}^{\min}\leq 1\leq E_{S}^{\max}$. We further say that $S$ is the generator of quantum noise evolution for $\nabla$ if

\begin{align}\label{EQ.DFN.L2W_Log_Mean_QNE_MaxEnt_3}
E_{S}^{\textrm{dis}}(\mu,t)=0
\end{align}

\begin{reapply}
\end{reapply}

\noindent for all $\mu\in\partial\SII(A)$ and $t\geq 0$.
\end{itemize}
\end{dfn}

Lemma \ref{LEM.L2W_Log_Mean_QNE_MaxEnt} gives equivalent conditions for minimising Equation \ref{EQ.DFN.L2W_Log_Mean_QNE_MaxEnt_2}. Note $4)$ in the lemma shows $-\Delta$ is the unique minimiser. Moreover, Corollary \ref{COR.L2W_Log_Mean_QNE_MaxEnt} gives maximal production of quantum entropy as per Equation \ref{EQ.COR.L2W_Log_Mean_QNE_MaxEnt_1} for normalised energy scales by maximising Equation \ref{EQ.SSEC.L2W_Log_Mean_QNE_5}. This gives our maximum entropy production principle in the finite-dimensional setting. Lemma \ref{LEM.L2W_Log_Mean_QNE_MaxEnt} ensures we do select $-\Delta$ as claimed. Max\-imisation constraints on energy spent in Equation \ref{EQ.COR.L2W_Log_Mean_QNE_MaxEnt_1} are indeed given by suitable evaluation of quantum Fisher information as per Definition \ref{DFN.QF_AF} at each state.\par


\pagebreak


\begin{lem}\label{LEM.L2W_Log_Mean_QNE_MaxEnt}
For all $S\in\BII(A)_{h}$ which are candidates for generating quantum noise evolution for $\nabla$, the following are equivalent:

\begin{itemize}
\item[1)] $S$ is the generator of quantum noise evolution for $\nabla$,

\item[2)] $\grad_{S}=0$ and $\var_{S}(\mu,t)=0$ for all $\mu\in\partial\SII(A)$ and $t\geq 0$,

\item[3)] $E_{S}(\mu,t)=1$ for all $\mu\in\partial\SII(A)$ and $t\geq 0$,

\item[4)] $S=-\Delta$.
\end{itemize}
\end{lem}
\begin{proof}
Let $\mathbf{S}$ be the set of all candidates for generating quantum noise evolution for $\nabla$. For all $S\in\mathbf{S}$, Equation \ref{EQ.SSEC.L2W_Log_Mean_QNE_6} and Equation \ref{EQ.DFN.L2W_Log_Mean_QNE_MaxEnt_3} show $S$ is the generator of quantum noise evolution for $\nabla$ if and only if

\begin{align}\label{EQ.LEM.L2W_Log_Mean_QNE_MaxEnt_1}
E_{S}^{\textrm{dis}}(\mu,t)=\inf_{S'\in\mathbf{S}}\hspace{0.025cm} E_{S'}^{\textrm{dis}}(\mu,t)=0=E_{-\Delta}^{\textrm{dis}}(\mu,t)
\end{align}

\noindent for all $\mu\in\partial\SII(A)$ and $t\geq 0$. Following our discussion of Equation \ref{EQ.SSEC.L2W_Log_Mean_QNE_10}, we know $-\Delta$ is the unique minimiser in our case. Using the latter, get $1)$ to $4)$.
\end{proof}

\begin{cor}\label{COR.L2W_Log_Mean_QNE_MaxEnt}
Let $S\in\BII(A)_{h}$ be a candidate for generating quantum noise evolution for $\nabla$. Then $S$ is the generator of quantum noise evolution for $\nabla$ if and only if for all fixed states $\xi\in\SII(A)$, we have

\begin{align}\label{EQ.COR.L2W_Log_Mean_QNE_MaxEnt_1}
-\restr{0.925}{\frac{d}{dt}}{t=0}\Enttau\big(G_{t}^{S}(\mu)\big)=\sup_{y\in\mathfrak{S}_{\xi,\mu}\hspace{0.025cm} (\mathfrak{H}_{\xi,\mu}(\sharp\mu))}\mathfrak{H}_{\xi,\mu}(y)=\sup_{y\in\mathfrak{S}_{\xi,\mu}(\mathcal{I}^{\log}(\mu,\mu,(\nabla\sharp\mu)^{\flat}))}\hspace{0.025cm} \mathfrak{H}_{\xi,\mu}(y)
\end{align}

\noindent for all $\mu\in\vartheta(\xi)$.
\end{cor}
\begin{proof}
Let $\xi\in\SII(A)$ be a fixed state and $\mu\in\vartheta(\xi)$. Proposition \ref{PRP.L2W_Log_Mean_NCDS} shows $\sharp\Theta\big(\mu,\lc\Delta\sharp\mu\rc^{\flat}\big)=\mathfrak{G}_{\mu}\mathrlap{\phantom{\mathfrak{F}}_{\mu}}\mathfrak{F}^{-1}\lc\mathfrak{F}_{\mu}\lc\log\sharp\mu\rc\rc{}=\nabla\sharp\mu$ by twice application. Using the latter, $2)$ in Proposition \ref{PRP.RM_II} lets us calculate

\begin{align}\label{EQ.COR.L2W_Log_Mean_QNE_MaxEnt_2}
\mathfrak{H}_{\xi,\mu}\lc\sharp\mu\rc{}=g_{\mu}^{\xi}\lc\lc\Delta\sharp\mu\rc^{\flat},\lc\Delta\sharp\mu\rc^{\flat}\rc{}=\mathcal{I}^{\log}\lc\mu,\mu,\Theta\big(\mu,\lc\Delta\sharp\mu\rc^{\flat}\big)\rc{}=\mathcal{I}^{\log}\lc\mu,\mu,\lc\nabla\sharp\mu\rc^{\flat}\rc{}.
\end{align}

\noindent Equation \ref{EQ.COR.L2W_Log_Mean_QNE_MaxEnt_2} shows the second identity in Equation \ref{EQ.COR.L2W_Log_Mean_QNE_MaxEnt_1}. Lemma \ref{LEM.L2W_Log_Mean_QNE_MaxEnt} shows $S$ is the generator of quantum noise evolution for $\nabla$ if and only if $S=-\Delta$. Proposition \ref{PRP.L2W_Log_Mean_QNE_MaxEnergy} thus implies the first identity in Equation \ref{EQ.COR.L2W_Log_Mean_QNE_MaxEnt_1}.
\end{proof}


\subsubsection*{Generators of quantum noise evolution}

Definition \ref{DFN.L2W_Log_Mean_QNE} gives our maximum entropy production principle. The fourth model assumption is locality therein. Following our discussion of the coarse graining process in Subsection \ref{SSEC.QOT_CG}, we justify locality as a natural complement to the first model assumption since Theorem \ref{THM.QOT_Minimiser_Approximation} lets us describe quantum optimal transport itself as scaling limit w.r.t.~the coarse graining process.\par


\pagebreak


We therefore view quantum Laplacians as generators of quantum noise evolution as per $1)$ in Theorem \ref{THM.L2W_Log_Mean_QNE}. Fittingly, $2)$ in Theorem \ref{THM.L2W_Log_Mean_QNE} shows quantum Laplacians satisfy, up to sign, a quantum Fokker-Planck equation with vanishing drift term in scaling limit, i.e.~only noise diffusion term. Thus $3)$ in Theorem \ref{THM.L2W_Log_Mean_QNE} shows heat flow satisfies a steepest entropy ascent property \cite{ART.Ber_Con_Mon.2015.MaxEnt_SEA} by considering the steepest descent property of gradient flows in smooth Riemannian manifolds \cite{BK.Lan.1995.Riemannian_Manifolds} as per Equation \ref{EQ.REM.L2W_Log_Mean_QNE_MaxEnergy_1} and taking limits. We thereby obtain slopes of maximal entropy production, i.e.~erasure of quantum information, as per Equation \ref{EQ.THM.L2W_Log_Mean_QNE_2} for the given subsets of all bounded normal states. If heat flow is $\EVI_{\lambda}$-gradient flow of quantum relative entropy, then Equation \ref{EQ.THM.L2W_Log_Mean_QNE_2} generalises to metric slopes as per Equation \ref{EQ.SSEC.L2W_EVI_Equivalence_1} for all normal states with finitely supported fixed part and finite quantum relative entropy. Note Remark \ref{REM.L2W_Log_Mean_QNE}.\par
Let $(\phi,\bpsi,\gamma,\nabla)$ be noncommutative differential structure for tracial AF-$C^{*}$-algebras $(A,\tau)$ and $(B,\omega)$ in the logarithmic mean setting.
    
\begin{dfn}\label{DFN.L2W_Log_Mean_QNE}
Let $S\in\UBII\lc{}L^{2}(A,\tau)\rc_{h}$ be local. We say that $S$ is the generator of quantum noise evolution for $\nabla$ if for all $j\in\mathbb{N}$, $S_{j}\in\BII(A_{j})_{h}$ is the generator of quantum noise evolution for $\nabla_{\hspace{-0.055cm} j}$.
\end{dfn}

\begin{thm}\label{THM.L2W_Log_Mean_QNE}
Let $(\phi,\bpsi,\gamma,\nabla)$ be noncommutative differential structure for tracial AF-$C^{*}$-algebras $(A,\tau)$ and $(B,\omega)$ in the logarithmic mean setting.

\begin{itemize}
\item[1)] $-\Delta$ is the generator of quantum noise evolution for $\nabla$.

\item[2)] For all $j\in\mathbb{N}$, let $\lc{}0,\varphi_{j},C_{j}\rc$ be a Lindblad decomposition of $-\Delta_{j}$. We have

\begin{align}\label{EQ.THM.L2W_Log_Mean_QNE_1}
-\Delta u=\|.\|_{\tau}\textrm{-}\lim_{j\in\mathbb{N}}\hspace{0.025cm} -\Delta_{j}u_{j}=\|.\|_{\tau}\textrm{-}\lim_{j\in\mathbb{N}}\hspace{0.025cm} \frac{C_{j}}{2}\lc{}2\varphi_{j}(u_{j})-\big\{\varphi_{j}(1_{A_{j}}),u_{j}\big\}\rc{}
\end{align}

\begin{reapply}
\end{reapply}

\noindent for all $u\in\dom\Delta$.

\item[3)] Let $\xi\in\SII(A)$ be a finitely supported fixed state. Let $p\in L^{1,\infty}(A,\tau)$ be a projection s.t.~we have $\xi\in\CI[p]$. For all $\mu\in\Fix_{A}^{\NI}(\xi)\cap\mathcal{S}_{-1}^{\NI,\infty}(A_{\xi})\cap\lc\dom\Delta\rc^{\flat}$, there exists maximal $\varepsilon\in (0,\infty]$ s.t.~

\begin{align}\label{EQ.THM.L2W_Log_Mean_QNE_2}
-\frac{d}{dt}\Enttau\lc{}h_{t}(\mu)\rc{}=\tau\lc\Delta\sharp h_{t}(\mu)\log\sharp h_{t}(\mu)\rc{}=\mathcal{I}^{\log}\lc{}h_{t}(\mu),h_{t}(\mu),\lc\nabla\sharp h_{t}(\mu)\rc^{\flat}\rc{}
\end{align}

\begin{reapply}
\end{reapply}

\noindent for all $t\in [0,\varepsilon)$.
\end{itemize}
\end{thm}
\begin{proof}
We show $1)$ and $2)$. Note $4.3)$ in Proposition \ref{PRP.Wstar_Derivation_QG_I} shows $A_{0}$ is core of $\Delta$. For all $j\in\mathbb{N}$, $1)$ in Proposition \ref{PRP.Wstar_Derivation_QG_II} shows $\Delta_{j}=\comAj\Delta\in\BII(A_{j})_{h}$. Thus $-\Delta$ is local, hence Lemma \ref{LEM.L2W_Log_Mean_QNE_MaxEnt} shows it is the generator of quantum noise evolution for $\nabla$. If we use $\|.\|_{\tau}$-limits as per $4.3)$ in Proposition \ref{PRP.Wstar_Derivation_QG_I}, then Equation \ref{EQ.THM.L2W_Log_Mean_QNE_1} is given by considering Equation \ref{EQ.SSEC.L2W_Log_Mean_QNE_1} for each $-\Delta_{j}$ and letting $j\uparrow\infty$. Altogether, get $1)$ and $2)$.\par
We show $3)$. Assume its setting. Let $\mu\in\Fix_{A}^{\NI}(\xi)\cap\mathcal{S}_{-1}^{\NI,\infty}(A_{\xi})\cap\lc\dom\Delta\rc^{\flat}$. Following $1)$ in Proposition \ref{PRP.Wstar_Derivation_QG_HSG_II} for $p=2$, the Hille-Yosida theorem applies to heat flow \lc{}cf.~p.79 in \cite{BK.Eng_Nag.2000.Semigroups}\rc{}. Using the latter, $3)$ in Proposition \ref{PRP.Wstar_Derivation_QG_HSG_Fixed_Part_I} and $2.2)$ in Theorem \ref{THM.Wstar_Derivation_QG_HSG_Regularity}, we show

\begin{align}\label{EQ.THM.L2W_Log_Mean_QNE_3}
h_{t}(\mu)\in\textrm{Fix}_{A}^{\NI}(A)\cap\mathcal{S}^{\NI,\infty}(A_{\xi})\cap\lc\dom\Delta\rc^{\flat}
\end{align}

\noindent for all $t\geq 0$. Note $\GL\lc{}L^{\infty}(A_{\xi},\tau)\rc\subset L^{\infty}(A_{\xi},\tau)$ open in norm topology. Using the latter and strong continuity as per $1)$ in Proposition \ref{PRP.Wstar_Derivation_QG_HSG_II}, we obtain $\varepsilon>0$ s.t.~

\begin{align}\label{EQ.THM.L2W_Log_Mean_QNE_4}
h_{t}(\mu)\in\mathcal{S}_{-1}^{\NI,\infty}(A_{\xi}) 
\end{align}

\noindent for all $t\in [0,\varepsilon]$. Equation \ref{EQ.THM.L2W_Log_Mean_QNE_3} and Equation \ref{EQ.THM.L2W_Log_Mean_QNE_4} show

\begin{align}\label{EQ.THM.L2W_Log_Mean_QNE_5}
h_{t}(\mu)\in\textrm{Fix}_{A}^{\NI}(A)\cap\mathcal{S}_{-1}^{\NI,\infty}(A_{\xi})\cap\lc\dom\Delta\rc^{\flat}
\end{align}

\noindent for all $t\in [0,\varepsilon]$. Note $\xi\in\CI[p]$. We have $\Fix_{A}^{\NI}(\xi)\subset\CI[p]$ by $1)$ in Theorem \ref{THM.QOT_Distance_AC_FS}. Using the latter and Corollary \ref{COR.Rel_Ent_AF_Cstar_Trace}, Equation \ref{EQ.THM.L2W_Log_Mean_QNE_3} implies

\begin{align}\label{EQ.THM.L2W_Log_Mean_QNE_6}
\Ent\lc{}h_{t}(\mu),\tau\rc{}=\tau\lc\sharp h_{t}(\mu)\log\sharp h_{t}(\mu)\rc{}=\tau\lc\comp\sharp h_{t}(\mu)\log\comp\sharp h_{t}(\mu)\rc{}
\end{align}

\noindent for all $t\in [0,\varepsilon]$. Since $\xi\in\dom\Enttau$, note $\xi\in\mathcal{S}^{\NI}(A)\cap\CI[p]$ by Lemma \ref{LEM.Rel_Ent_AF_Cstar_Trace_II}. We have $\supp\xi\leq p$ by Lemma \ref{LEM.AF_Support_Projection_Majorant_Uniform}. Equation \ref{EQ.THM.L2W_Log_Mean_QNE_5} shows we may replace $p$ with $\supp\xi$ in Equation \ref{EQ.THM.L2W_Log_Mean_QNE_6}. Using the latter and $4.2)$ in Corollary \ref{COR.Wstar_Derivation_QG_HSG_Regularity}, Equation \ref{EQ.THM.L2W_Log_Mean_QNE_5} implies we have Fr\'echet differentiable map $t\mapsto\sharp h_{t}(\mu)\in L^{\infty}(A_{\xi},\tau)_{>0}\cap L^{\infty}(A_{\xi},\tau)_{\nabla}$ defined on $[0,\varepsilon)$.\par
We thereby extend calculations in Lemma \ref{LEM.L2W_Log_Mean_NCDS} and Corollary \ref{COR.L2W_Log_Mean_NCDS}, in particular those involving operator differentiable functions \cite{ART.Ped.2000.OpAlg_Diff_Functions}, to our setting. Lemma \ref{LEM.AF_Support_Projection_Majorant_Uniform} shows $\xi$ has integrable support. Using Proposition \ref{PRP.L2W_Log_Mean_NCDS}, we calculate

\begin{align*}
-\frac{d}{dt}\Enttau\lc{}h_{t}(\mu)\rc{} & = \tau\lc\Delta\sharp h_{t}(\mu)\log\sharp h_{t}(\mu)\rc{}+\tau\lc\sharp h_{t}(\mu)d\log_{\sharp h_{t}(\mu)}\lc\Delta\sharp h_{t}(\mu)\rc\rc \phantom{\Bigg)} \\
& = \tau\lc\Delta\sharp h_{t}(\mu)\log\sharp h_{t}(\mu)\rc \phantom{\Bigg)} \\
& = \lgl\mathcal{D}_{\sharp h_{t}(\mu),\xi}\nabla\sharp h_{t}(\mu),\nabla\sharp h_{t}(\mu)\rgl_{\omega} \phantom{\Bigg)} \\
& = \mathcal{I}^{\log}\lc{}h_{t}(\mu),h_{t}(\mu),\lc\nabla\sharp h_{t}(\mu)\rc^{\flat}\rc \phantom{\Bigg)}
\end{align*}

\noindent for all $[0,\varepsilon)$. Note Remark \ref{REM.L2W_Log_Mean_NCDS_I}. The above calculation shows Equation \ref{EQ.THM.L2W_Log_Mean_QNE_2}. Since $\varepsilon>0$ by construction, there exists maximal choice of $\varepsilon\in (0,\infty]$ as claimed.
\end{proof}

\begin{rem}\label{REM.L2W_Log_Mean_QNE}
Note $3)$ in Theorem \ref{THM.L2W_Log_Mean_QNE} generalises Corollary \ref{COR.L2W_Log_Mean_NCDS} by the semigroup property. Using standard arguments for interchanging derivatives and limits \cite{BK.Eva.2010.Partial_Differential_Equations}\cite{BK.Koe.1993.Analysis_II}\linebreak\cite{BK.Koe.2004.Analysis_I}, we see $3.2)$ in Theorem \ref{THM.QOT_Distance_AC_FS} and $3)$ in Theorem \ref{THM.L2W_Log_Mean_QNE} let us calculate

\begin{align*}
-\frac{d}{dt}\Enttau\lc{}h_{t}(\mu)\rc{} & = \lim_{j\in\mathbb{N}}\hspace{0.025cm} -\frac{d}{dt}\Enttau\lc{}h_{t}\lc\bar{\mu}_{j}\rc\rc \phantom{\Bigg)} \\
& =\lim_{j\in\mathbb{N}}\hspace{0.025cm} \mathcal{I}^{\log}\lc{}h_{t}\lc\bar{\mu}_{j}\rc{},h_{t}\lc\bar{\mu}_{j}\rc{},\lc\nabla\sharp h_{t}\lc\bar{\mu}_{j}\rc\rc^{\flat}\rc \phantom{\Bigg)} \\
& =\mathcal{I}^{\log}\lc{}h_{t}(\mu),h_{t}(\mu),\lc\nabla\sharp h_{t}(\mu)\rc^{\flat}\rc \phantom{\Bigg)}
\end{align*}

\noindent for all $\mu\in\mathcal{S}^{\NI,2}(A)\cap \lc\dom\Delta\rc^{\flat}$ and $t>0$ if uniform convergence in the second identity is given in each case and to finite terms. We might thereby extend Equation \ref{EQ.THM.L2W_Log_Mean_QNE_2} to a maximal definition domain by means of coarse graining. Here, we do not consider such assumptions on uniform convergence since we do not know of any examples.
\end{rem}

Example \ref{BSP.L2W_Log_Mean_QNE_QG_Internal} gives the depolarising channel as canonical choice of quantum noise operator \lc{}cf.~pp.378-379 in \cite{BK.Nie_Chu.2000.Quantum_Computation_Information}\rc{}. We see internal quantum gradients induce quantum Laplacians which are, up to sign, infinitesimal applications of depolarising channels. This shows quantum Fokker-Planck equations with vanishing drift term in scaling limit as per Equation \ref{EQ.THM.L2W_Log_Mean_QNE_1} may have closed form description.

\begin{bsp}\label{BSP.L2W_Log_Mean_QNE_QG_Internal}
Assume $(A,\tau)$ is a strongly unital tracial AF-$C^{*}$-algebra s.t.~$\tau<\infty$, as\linebreak well as $(B,\omega)=(A\otimes A,\tau\otimes\tau)$ equipped with the internal AF-$A$-bimodule structure on $A\otimes A$ as per $1)$ in Definition \ref{DFN.Wstar_Derivation_QG_Internal}. Let $\lambda\in [0,1]$. We consider the $\lambda$-internal quantum gradient $\nabla^{\lambda}:A_{0}\longrightarrow L^{2}(A\otimes A,\tau\otimes\tau)$ on $A$ as per $2)$ in Definition \ref{DFN.Wstar_Derivation_QG_Internal}. We therefore have quantum Laplacian $\Delta^{\lambda}=\lambda\pi_{\ker\tau}^{A}\in\BII\lc{}L^{2}(A,\tau)\rc_{h}$ by $4)$ in Proposition \ref{PRP.Wstar_Derivation_QG_Internal}.\par
We obtain $-\Delta^{\lambda}\neq 0$, $-\Delta^{\lambda}\lc{}L^{\infty}(A,\tau)\rc\subset L^{\infty}(A,\tau)$ and $-\Delta^{\lambda}1_{A}=0$. We have completely Markovian semigroup $h:[0,\infty)\longrightarrow\BII\lc{}L^{\infty}(A,\tau)\rc$ by $1)$ in Proposition \ref{PRP.Wstar_Derivation_QG_HSG_II}. Using the latter, Lemma \ref{LEM.Wstar_CP_Markovian_SG} shows $-\Delta^{\lambda}$ has Lindblad decomposition. We show $-\Delta^{\lambda}$ satisfies a quantum Fokker-Planck equation with vanishing drift term. We define the depolarising channel $\varphi^{\lambda}:L^{\infty}(A,\tau)\longrightarrow L^{\infty}(A,\tau)$ with depolarisation probability $\lambda$ by setting

\begin{align}\label{EQ.L2W_Log_Mean_QNE_QG_Internal_1}
\varphi^{\lambda}(x):=\lc{}1-\lambda\rc{}x+\lambda\vstretch{1.375}{\big(}I-\pi_{\ker\tau}^{A}\vstretch{1.375}{\big)}(x)=\vstretch{1.375}{\big(}I-\lambda\pi_{\ker\tau}^{A}\vstretch{1.375}{\big)}(x)=\lc{}I-\Delta^{\lambda}\rc{}(x)
\end{align}

\noindent for all $x\in L^{\infty}(A,\tau)$ \lc{}cf.~pp.378-379 in \cite{BK.Nie_Chu.2000.Quantum_Computation_Information}\rc{}. Moreover, we directly verify all maps of form $x\mapsto C\tau(x)1_{A}$ defined on $L^{\infty}(A,\tau)$ for $C>0$ are completely positive. Yet

\begin{align}\label{EQ.L2W_Log_Mean_QNE_QG_Internal_2}
\vstretch{1.375}{\big(}I-\pi_{\ker\tau}^{A}\vstretch{1.375}{\big)}(x)=\frac{\tau(x)}{\tau(1_{A})}1_{A}
\end{align}

\noindent for all $x\in L^{\infty}(A,\tau)$.\par
Equation \ref{EQ.L2W_Log_Mean_QNE_QG_Internal_2} shows $\varphi^{\lambda}:A\longrightarrow L^{\infty}(A,\tau)$ is a completely positive trace-preserving unital map. Following Remark \ref{REM.Wstar_CP_Markovian_SG}, Equation \ref{EQ.L2W_Log_Mean_QNE_QG_Internal_1} shows $\varphi^{\lambda}$ is the quantum channel transmitting change of states given by complete mixing with uniform probability $\lambda$ for all states. Using $\varphi^{\lambda}=I-\Delta^{\lambda}$, we calculate

\begin{align}\label{EQ.L2W_Log_Mean_QNE_QG_Internal_3}
-\Delta^{\lambda}x=-\lc{}I-\varphi^{\lambda}\rc{}(x)=\frac{1}{2}\lc{}2\varphi^{\lambda}(x)-\big\{\varphi^{\lambda}(1_{A}),x\big\}\rc{}
\end{align}

\noindent for all $x\in L^{\infty}(A,\tau)$. Equation \ref{EQ.L2W_Log_Mean_QNE_QG_Internal_3} yields Lindblad decomposition $\lc{}0,\varphi^{\lambda},1\rc$ of $-\Delta^{\lambda}$ and closed form of Equation \ref{EQ.THM.L2W_Log_Mean_QNE_1}. Following Remark \ref{REM.Wstar_CP_Markovian_SG}, Equation \ref{EQ.L2W_Log_Mean_QNE_QG_Internal_3} is a quantum Fokker-Planck equation with vanishing drift term. If we do not use fixed energy scales normalised as per Equation \ref{EQ.SSEC.L2W_Log_Mean_QNE_11}, then Equation \ref{EQ.L2W_Log_Mean_QNE_QG_Internal_3} does not suffice to determine unique quantum noise evolution even as the depolarising probability itself is fixed.
\end{bsp}


\section{EVI$_{\lambda}$-gradient flow of quantum relative entropy}\label{SEC.L2W_EVI}

We emulate the classical case in the infinitesimally Hilbertian setting \cite{ART.Erb_Kuw_Stu.2015.Classical_OT_Equivalence}. Following work of Jordan, Kinderlehrer and Otto for Fokker-Planck equations \cite{ART.Jor_Kin_Ott.1998.Fokker_Planck}, resp.~Otto for porous medium equations \cite{ART.Ott.2001.Classical_OT_Porous_Medium}\cite{ART.Ott.2005.Classical_OT_GradFlow_DisConvex}, Ambrosio, Gigli and Savar\'e give $\EVI_{\lambda}$-gradient flows of proper l.s.c.~functionals defined on metric spaces \cite{BK.Amb_Gig_Sav.2008.Classical_OT_GradFlow} to study evolution partial differential equations using gradient flows absent differential structures \cite{ART.Dan_Sav.2008.Classical_OT_GradFlow_DisConvex}\cite{ART.Mur_Sav.2020.Classical_OT_EVI}. Note $\EVI_{\lambda}$-gradient flows generalise gradient flows in smooth Riemannian manifolds driven by smooth functionals with Hessians bounded from below \cite{BK.Amb_Gig_Sav.2008.Classical_OT_GradFlow}\cite{ART.Erb.2010.Weak_Riemannian_structure}\cite{ART.Mur_Sav.2020.Classical_OT_EVI}. We therefore apply results in variational analysis for metric geometry using minimising geodesics \cite{ART.Dan_Sav.2008.Classical_OT_GradFlow_DisConvex}\cite{ART.Mur_Sav.2020.Classical_OT_EVI} to study quantum relative entropy in the logarithmic mean setting.\par
Analogous $L^{2}$-Wasserstein distances in the classical case \cite{ART.Dol_Naz_Sav.2009.Generalised_OT} are those determined by weak upper gradients \cite{ART.Amb_Mar_Sav.2014.Equivalence_Gradients}\cite{ART.Chee.1999.Relaxed_Gradients} inducing Dirichlet forms \cite{BK.Fuk_Osh_Tak.2011.Dirichlet_Markov}. If $\EVI_{\lambda}$-gradient flow of relative entropy exists in this case, then it is heat flow \cite{ART.Amb_Gig_Sav.2014.Classical_OT_Ricci_Bounds_I}\cite{ART.Amb_Gig_Sav.2014.Classical_OT_Ricci_Bounds_II}. Existence is equivalent to $\lambda$-convexity of relative entropy \cite{ART.Amb_Gig_Sav.2014.Classical_OT_Ricci_Bounds_I}\cite{ART.Amb_Gig_Sav.2014.Classical_OT_Ricci_Bounds_II} and Bakry-\'Emery conditions \cite{COL.Bak_Em.1985.Hypercontractivity_Condition}\cite{ART.Bak_Led.2006.Hypercontractivity_Condition} linking heat flow to a weak Riemannian structure \cite{BK.Amb_Gig_Sav.2008.Classical_OT_GradFlow}\cite{ART.Erb.2010.Weak_Riemannian_structure} for the given classical $L^{2}$-Wasserstein dis\-tance \cite{ART.Amb_Gig_Sav.2015.Classical_OT_Ricci_Bounds_III}\cite{ART.Amb_Mon_Sav.2019.Classical_OT_Nonlinear_Diffusion}\cite{ART.Erb_Kuw_Stu.2015.Classical_OT_Equivalence}. Sturm \cite{ART.Stu.2006.Classical_OT_I}\cite{ART.Stu.2006.Classical_OT_II}, as well as Lott and Villani \cite{ART.Lot_Vil.2009.Classical_OT_Ricci_Bounds}, each established $\lambda$-convexity of relative entropy \cite{ART.CorEra_McCan_Sch.2001.Displacement_Convexity_Riemannian}\cite{ART.McCan.1997.Displacement_Convexity_Local} as synthetic lower Ricci bounds \cite{ART.Ren_Stu.2005.Smooth_OT_Equivalences}. They imply following chain of functional inequalities \cite{ART.Lot_Vil.2009.Classical_OT_Ricci_Bounds}\cite{ART.Ott_Vil.2000.Classical_OT_LogSobolev_Talagrand} probing the underlying metric geometry. As for Riemannian manifolds \cite{ART.Fig_Vil.2007.Equivalence_Convexity}\cite{ART.Ren_Stu.2005.Smooth_OT_Equivalences}, there exists a $\HWI_{\lambda}$-interpolation inequality and Talagrand inequality $\TW_{\lambda}$ for $\lambda\geq 0$, and a modified logarithmic Sobolev inequality $\MLSI_{\lambda}$ for $\lambda>0$. If we do have lower Ricci bound $\lambda>0$, then $\lambda$-convexity of relative entropy implies $\HWI_{\lambda}$, in turn implying $\MLSI_{\lambda}$, finally implying $\TW_{\lambda}$ \cite{ART.Lot_Vil.2009.Classical_OT_Ricci_Bounds}. If we want lower Ricci bounds and functional inequalities for quantum $L^{2}$-Wasserstein distances in direct analogy to the classical case, then we initially require equivalent characterisations for $\EVI_{\lambda}$-gradient flow of quantum relative entropy in the logarithmic mean setting. Since we cover all fundamental example classes in Subsection \ref{SSEC.QOT_DT_BSP}, we also face complications arising from non-ergodicity commonly avoided in the classical case by assumption. We cannot expect ergodicity in the AF-$C^{*}$-setting because dynamic quantum gradients generalise the ubiquitous case of inner derivations \cite{ART.Kad.1966.OpAlg_Derivations}.\par
In the ergodic finite-dimensional logarithmic mean setting, Carlen and Maas extend equivalent characterisations and functional inequalities \cite{ART.Car_Maa.2014.Quantum_OT_I}\cite{ART.Car_Maa.2017.Quantum_OT_II}\cite{ART.Car_Maa.2020.Quantum_OT_III}. Equivalence in \cite{ART.Car_Maa.2020.Quantum_OT_III} uses arguments fully given by Erbar and Maas in \cite{ART.Erb_Maa.2012.Discrete_OT_Ricci_Bounds} alone. We use \cite{ART.Car_Maa.2020.Quantum_OT_III} as foundation and apply the coarse graining process to reduce to the finite-dimensional Riemannian setting s.t.~ergodicity is recovered. We extend results upon replacing some essential arguments used in \cite{ART.Car_Maa.2020.Quantum_OT_III} and \cite{ART.Erb_Maa.2012.Discrete_OT_Ricci_Bounds}. Ours and independent work of Wirth \cite{PRE.Wir.2018.NC_OT} together with Zhang \cite{ART.Wir_Zha.2021.Quantum_OT_Complete_Gradient_Estimates} are the first infinite-dimensional extensions of the results in \cite{ART.Car_Maa.2014.Quantum_OT_I}\cite{ART.Car_Maa.2017.Quantum_OT_II}\cite{ART.Car_Maa.2020.Quantum_OT_III}. Wirth \cite{PRE.Wir.2018.NC_OT} gives noncommutative optimal transport distances determined by suitable symmetric $C^{*}$-derivations inducing $C^{*}$-Dirichlet forms on noncommutative $L^{2}$-spaces of tracial $W^{*}$-algebras \cite{ART.Cip.1997.NC_Dirichlet_Markov}\cite{ART.Cip_Sav.2003.NC_Dirichlet_Grad}. Assuming tracial state and ergodicity, Wirth shows a, possibly infinite-dimensional, Bakry-\'Emery condition \cite{PRE.Wir.2018.NC_OT} as per \cite{ART.Car_Maa.2020.Quantum_OT_III} implies heat flow is $\EVI_{\lambda}$-gradient flow of relative entropy for $W^{*}$-algebras \cite{BK.Ohy_Pet.1993.Rel_Ent}. However, \cite{PRE.Wir.2018.NC_OT} does not show its equivalence. Assuming tracial state, Wirth and Zhang give sufficient conditions for satisfying Bakry-\'Emery conditions \cite{ART.Wir_Zha.2021.Quantum_OT_Complete_Gradient_Estimates} as per \cite{ART.Car_Maa.2020.Quantum_OT_III} and obtain functional inequalities $\HWI_{\lambda}$, $\MLSI_{\lambda}$ \cite{ART.Wir_Zha.2021.Quantum_OT_Complete_Gradient_Estimates} and $\TW_{\lambda}$ \cite{PRE.Wir.2018.NC_OT} as per \cite{ART.Car_Maa.2020.Quantum_OT_III} using relative entropy for $W^{*}$-algebras conditioned to fixed-point subalgebras. Such a priori conditioning handles non-ergodicity but does not emerge from an underlying metric geometry. As part of the overall introduction, we show their results are insufficient for our purposes.\par
We contribute the following. In our logarithmic mean setting, which does assume the AF-$C^{*}$-setting, yet neither ergodicity nor finite trace, we extend results in \cite{ART.Car_Maa.2014.Quantum_OT_I}\cite{ART.Car_Maa.2017.Quantum_OT_II}\cite{ART.Car_Maa.2020.Quantum_OT_III} and \cite{ART.Erb_Maa.2012.Discrete_OT_Ricci_Bounds} to the general case and view lower Ricci bounds as measurement convexity of quantum information. Non-ergodicity and non-finite trace ensure fundamental example classes in Subsection \ref{SSEC.QOT_DT_BSP} are covered. We extend results in four parts by means of the coarse graining process. This lets us modify arguments in \cite{ART.Car_Maa.2020.Quantum_OT_III} and \cite{ART.Erb_Maa.2012.Discrete_OT_Ricci_Bounds} for the known ergodic finite-dimensional case. First, we show claimed equivalence of $\EVI_{\lambda}$-gradient flow, $\lambda$-convexity, Bakry-\'Emery and Hessian lower bound conditions. Secondly, we then define lower Ricci bounds of quantum gradients. Thirdly, we give sufficient conditions for lower Ricci bounds of direct sum quantum gradients. Fourthly, we derive functional inequalities in the AF-$C^{*}$-setting. This requires quantum Fisher information. Apart from extension and following our view of quantum Laplacians as generators of quantum noise evolution in Subsection \ref{SSEC.L2W_Log_Mean_QNE}, lower Ricci bounds are given by $\lambda$-convexity of quantum information along minimising geodesics measured by quantum relative entropy. If we have noncommutative analogues of displacement interpolations \cite{ART.CorEra_McCan_Sch.2001.Displacement_Convexity_Riemannian}\cite{ART.McCan.1997.Displacement_Convexity_Local}, then such measurement convexity in the Schr\"odinger picture is convexity under measurement of observables in the Heisenberg picture. Unfortunately, existence results are unknown to us. We instead show strictly positive lower Ricci bounds determine energy-information trade-offs pa\-ra\-metrised by lower bounds on quantum noise. Lower resolution implies lower energy paths. We avoid spatial interpretations of the classical case \cite{ART.Dol_Naz_Sav.2009.Generalised_OT}\cite{ART.Lot_Vil.2009.Classical_OT_Ricci_Bounds}.

\medskip

\noindent\textbf{Structure.} In Subsection \ref{SSEC.L2W_EVI_Equivalence}, we discuss $\EVI_{\lambda}$-gradient flows in metric spaces and $\lambda$-convexity of proper l.s.c.~functionals. We consider heat flow as $\EVI_{\lambda}$-gradient flow of quantum relative entropy and show our equivalence theorem. In Subsection \ref{SSEC.L2W_EVI_Ric}, we discuss lower Ricci bounds, energy-information trade-offs parametrised by lower bounds on quantum noise and derive functional inequalities.\par


\newpage



\subsection{The equivalence theorem}\label{SSEC.L2W_EVI_Equivalence}

Following our discussion of the coarse graining process in Subsection \ref{SSEC.QOT_CG}, we define the $\EVI_{\lambda}$-gradient flow, $\lambda$-convexity and Bakry-\'Emery conditions in global and local form. We furthermore consider a Hessian lower bound condition. Such conditions are, in their global form, properties on all finitely supported accessibility components having non-trivial intersection with the domain of quantum relative entropy. Non-ergodicity requires us to consider accessibility components. Compatibility with compression and finite-dimensional approximation requires finitely supported ones. Theorem \ref{THM.L2W_EVI_Equivalence} shows equivalence of all conditions by means of the coarse graining process.


\subsubsection*{$\EVI_{\lambda}$-gradient flows and $\lambda$-convexity}

Metric-functional systems provide the general setting \cite{BK.Amb_Gig_Sav.2008.Classical_OT_GradFlow}\cite{ART.Mur_Sav.2020.Classical_OT_EVI}. Definition \ref{DFN.Metric_Functional_EVI_CNV} gives $\EVI_{\lambda}$-gradient flows and $\lambda$-convexity as per \cite{BK.Amb_Gig_Sav.2008.Classical_OT_GradFlow}\cite{ART.Mur_Sav.2020.Classical_OT_EVI}. We use continuous semigroups on metric spaces \cite{BK.Amb_Gig_Sav.2008.Classical_OT_GradFlow} generalising those on Banach spaces \cite{BK.Eng_Nag.2000.Semigroups}. Note $2)$ in Definition \ref{DFN.Metric_Functional_EVI_CNV} is called strong geodesic $\lambda$-convexity if it is to be distinguished from weaker formulations. We only use the former and therefore call it $\lambda$-convexity throughout our discussion. We use minimising geodesics defined on the unit interval \cite{BK.Amb_Gig_Sav.2008.Classical_OT_GradFlow}\cite{BK.Bur_Bur_Iv.2001.Metric_Geometry}. Proposition \ref{PRP.Metric_Functional_EVI_CNV} collects properties of $\EVI_{\lambda}$-gradient flows.\par
We review gradient flows in metric spaces determined by evolution variational inequalities, or $\EVI_{\lambda}$-gradient flows. Let $\lc{}X,d\rc$ be a complete geodesic length-metric space and $F:X\longrightarrow (-\infty,\infty]$ a proper functional l.s.c.~in $d$-topology. Let $Y\subset\overline{\dom F}$ s.t.~for all $\mu^{0},\mu^{1}\in Y\cap\dom F$, there exists minimising geodesic $\mu:[0,1]\longrightarrow Y\cap\dom F$ from $\mu^{0}$ to $\mu^{1}$. Let $\lambda\in\mathbb{R}$. If $S:[0,\infty)\times Y\longrightarrow Y$ is $\EVI_{\lambda}$-gradient flow as per $1)$ in Definition \ref{DFN.Metric_Functional_EVI_CNV}, then it is $\lambda$-contracting as per $1)$ in Proposition \ref{PRP.Metric_Functional_EVI_CNV} and its curves are of maximal slope by Theorem 3.5 in \cite{ART.Mur_Sav.2020.Classical_OT_EVI}, i.e.~each $t\mapsto S_{t}(\mu)$ is absolutely continuous and satisfies

\begin{align}\label{EQ.SSEC.L2W_EVI_Equivalence_1}
\frac{d^{+}}{dt}F\lc{}S_{t}(\mu)\rc{}=-\absv{1.15}{\partial F}^{2}\lc{}S_{t}(\mu)\rc{}
\end{align}

\noindent for a.e.~$t\geq 0$. We use metric slope $\mu\mapsto\absv{1.15}{\partial F}(\mu)$ of $F$ \cite{BK.Amb_Gig_Sav.2008.Classical_OT_GradFlow}\cite{ART.Mur_Sav.2020.Classical_OT_EVI}. Equation \ref{EQ.SSEC.L2W_EVI_Equivalence_1} recovers the steepest descent property of gradient flows in smooth Riemannian manifolds \cite{BK.Lan.1995.Riemannian_Manifolds}. Note existence of $\EVI_{\lambda}$-gradient flows implies $\lambda$-convexity of $F$ as per $2)$ in Definition \ref{DFN.Metric_Functional_EVI_CNV} by $4)$ in Theorem 3.10 in \cite{ART.Mur_Sav.2020.Classical_OT_EVI}. The chain rule of Riemannian metrics involving covariant derivatives \cite{BK.Lan.1995.Riemannian_Manifolds} implies $\lambda$-convexity generalises lower bounds for Hessians of smooth functionals \cite{ART.Ott.2005.Classical_OT_GradFlow_DisConvex}. Theorem 4.2 in \cite{ART.Mur_Sav.2020.Classical_OT_EVI} conversely shows $\lambda$-convexity of $F$ implies $S$ is the unique $\EVI_{\lambda}$-gradient flow given by the generalised minimising movements scheme \cite{COL.DeG.1993.Min_Mov}. Altogether, $\EVI_{\lambda}$-gradient flows generalise gradient flows in smooth Riemannian manifolds driven by smooth functionals with Hessians bounded from below.\par
If $\EVI_{\lambda}$-gradient flow of quantum relative entropy exist, then Corollary \ref{COR.L2W_EVI_Equivalence} shows it is heat flow. We further generalise slopes of maximal entropy production, i.e.~erasure of quantum information, as per Equation \ref{EQ.THM.L2W_Log_Mean_QNE_2} to Equation \ref{EQ.SSEC.L2W_EVI_Equivalence_1} for all normal states with finitely supported fixed part and finite quantum relative entropy as claimed in the introduction of Section \ref{SEC.L2W_Log_Mean}. We avoid the extension problems in Remark \ref{REM.L2W_Log_Mean_QNE}.

\begin{dfn}\label{DFN.Metric_Functional_EVI_CNV}
Let $\lc{}X,d\rc$ be a metric space s.t.~$d<\infty$ and $F:X\longrightarrow (-\infty,\infty]$ a proper functional l.s.c.~in $d$-topology. We call $\lc{}X,d,F\rc$ a metric-functional system. Let $Y\subset\overline{\dom F}$ and $\lambda\in\mathbb{R}$.

\begin{itemize}
\item[1)] We say that a continuous semigroup $S:[0,\infty)\times Y\longrightarrow Y$ is $\EVI_{\lambda}$-gradient flow of $F$ in $Y$ if for all $\mu\in Y$ and $\eta\in\dom F$, we have 

\begin{align}\label{EQ.DFN.Metric_Functional_EVI}
\frac{1}{2}\frac{d^{+}}{dt}d\lc{}S_{t}(\mu),\eta\rc^{2}+\frac{\lambda}{2}d\lc{}S_{t}(\mu),\eta\rc^{2}\leq F(\eta)-F\lc{}S_{t}(\mu)\rc{} \tag{$\EVI_{\lambda}$}
\end{align}

\begin{reapply}
\end{reapply}

\noindent for all $t\geq 0$.

\item[2)] Assume $\lc{}X,d\rc$ is a complete geodesic length-metric space. We call $Y\cap\dom F\subset X$ a geodesic subspace if for all $\mu^{0},\mu^{1}\in Y\cap\dom F$, there exists a minimising geodesic $\mu:[0,1]\longrightarrow X$ from $\mu^{0}$ to $\mu^{1}$ s.t.~we have $\mu(t)\in Y\cap\dom F$ for all $t\in [0,1]$. Assume $Y\cap\dom F\subset X$ is a geodesic subspace. We say that $F$ is $\lambda$-convex in $Y$ if for all minimising geodesics $\mu:[0,1]\longrightarrow Y\cap\dom F$, we have

\begin{align}\label{EQ.DFN.Metric_Functional_CNV}
F\lc\mu(t)\rc\leq \lc{}1-t\rc{}F\lc\mu(0)\rc{}+tF\lc\mu(1)\rc{}-\frac{\lambda}{2}t\lc{}1-t\rc{}d\lc\mu(0),\mu(1)\rc^{2} \tag{$\CNV_{\lambda}$}
\end{align}

\begin{reapply}
\end{reapply}

\noindent for all $t\in [0,1]$.
\end{itemize}
\end{dfn}

\begin{rem}\label{REM.Metric_Functional_EVI_CNV}
We have following integral characterisation of $\EVI_{\lambda}$-gradient flows. Let $\lc{}X,d,F\rc$ be a metric-functional system, $Y\subset\overline{\dom F}$ and $\lambda\in\mathbb{R}$. Theorem 3.3 in \cite{ART.Mur_Sav.2020.Classical_OT_EVI} shows a continuous semigroup $S:[0,\infty)\times Y\longrightarrow Y$ is $\EVI_{\lambda}$-gradient flow of $F$ in $Y$ if and only if for all $\mu\in Y$ and $\eta\in\dom F$, the map $t\mapsto F\lc{}S_{t}(\mu)\rc$ is strictly decreasing and we have

\begin{align}\label{EQ.REM.Metric_Functional_EVI_CNV_1}
\frac{e^{\lambda\lc{}t-s\rc{}}}{2}d\lc{}S_{t}(\mu),\eta\rc^{2}-\frac{1}{2}d\lc{}S_{s}(\mu),\eta\rc^{2}\leq\int_{0}^{t-s}e^{\lambda r}dr\cdot \bigg(F(\eta)-F\lc{}S_{t}(\mu)\rc\bigg) \tag{$\EVI_{\lambda}^{\int}$}
\end{align}

\noindent for all $0<s<t<\infty$.
\end{rem}

\begin{prp}\label{PRP.Metric_Functional_EVI_CNV}
Let $\lc{}X,d,F\rc$ be a metric-functional system, $Y\subset\overline{\dom F}$ and $\lambda\in\mathbb{R}$.  Let $S:[0,\infty)\times Y\longrightarrow Y$ be an $\EVI_{\lambda}$-gradient flow of $F$ in $Y$.

\begin{itemize}
\item[1)] For all $\mu,\eta\in Y$, we have 

\begin{align}\label{EQ.PRP.Metric_Functional_EVI_CNV_1}
d\lc{}S_{t}(\mu),S_{t}(\eta)\rc\leq e^{-\lambda t}d(\mu,\eta)
\end{align}

\begin{reapply}
\end{reapply}

\noindent for all $t\geq 0$.

\item[2)] Assume $F:X\longrightarrow (-\infty,\infty)$ has complete sublevels in $d$-topology. If $\lambda>0$, then $F$ has a unique minimum $\mu_{\min}\in Y\cap\dom F$.

\item[3)] Assume $\lc{}X,d\rc$ is a complete length-metric space. If $Y\cap\dom F\subset X$ is a geodesic subspace, then $F$ is $\lambda$-convex in $Y$.
\end{itemize}
\end{prp}


\pagebreak


\begin{proof}
The statement on $\lambda$-contraction as per Theorem 3.5 in \cite{ART.Mur_Sav.2020.Classical_OT_EVI} for $s=0$ shows $1)$ at once. Moreover, the statements on asymptotic behaviour as $t\rightarrow\infty$ as per Theorem 3.5 in \cite{ART.Mur_Sav.2020.Classical_OT_EVI} for $\lambda>0$ imply $2)$ by rearranging terms. Finally, note $4)$ in Theorem 3.10 in \cite{ART.Mur_Sav.2020.Classical_OT_EVI} implies $3)$ if $Y\cap\dom F\subset X$ is a geodesic subspace since the latter ensures existence of minimising geodesics.
\end{proof}


\subsubsection*{Equivalence in the logarithmic mean setting}

Following our discussion of the coarse graining process in Subsection \ref{SSEC.QOT_CG}, Definition \ref{DFN.L2W_EVI_Equivalence} gives the $\EVI_{\lambda}$-gradient flow, $\lambda$-convexity and Bakry-\'Emery conditions in global and local form. We furthermore consider a Hessian lower bound condition. In the finite-dimensional logarithmic mean setting, Lemma \ref{LEM.L2W_EVI_Equivalence} shows all conditions are equivalent. We are motivated in our proof by analogous arguments in \cite{ART.Car_Maa.2020.Quantum_OT_III} and \cite{ART.Erb_Maa.2012.Discrete_OT_Ricci_Bounds}. However, Theorem \ref{THM.L2W_Log_Mean_NCDS_Fin_Hessian} replaces essential steps therein letting us argue using Riemannian metrics on relative interiors induced by the given quasi-entropy. Theorem \ref{THM.L2W_EVI_Equivalence} uses Lemma \ref{LEM.L2W_EVI_Equivalence} to show equivalence of all conditions by means of the coarse graining process. Corollary \ref{COR.L2W_EVI_Equivalence} shows restriction to finitely supported accessibility components suffices to determine $\EVI_{\lambda}$-gradient flows.\par
Let $(\phi,\bpsi,\gamma,\nabla)$ be noncommutative differential structure for tracial AF-$C^{*}$-algebras $(A,\tau)$ and $(B,\omega)$ in the logarithmic mean setting. Notation \ref{NTN.L2W_EVI_Equivalence} summarises use of $2)$ in Theorem \ref{THM.QOT_Distance} and Lemma \ref{LEM.Rel_Ent_AF_Cstar_Trace_I}. Proposition \ref{PRP.L2W_EVI_Equivalence} gives metric-functional systems equipped with restricted $h:[0,\infty)\times\SII(A)\longrightarrow\SII(A)$ as continuous semigroup. We use these in Definition \ref{DFN.L2W_EVI_Equivalence}. 

\begin{ntn}\label{NTN.L2W_EVI_Equivalence}
Let $\xi\in\SII(A)$ be a finitely supported fixed state. Let $\mathcal{C}\subset (\SII(A),\mathcal{W}_{\nabla}^{\log})$ be finitely supported with fixed part $\xi$. For all $j\in\mathbb{N}$ s.t.~$\xi_{j}\neq 0$, we have

\begin{align}\label{EQ.NTN.L2W_EVI_Equivalence_1}
\restr{0.925}{\mathcal{W}_{\nabla}^{\log}}{\mathcal{C}_{A_{j}}\lc\bar{\xi}_{j}\rc\times\mathcal{C}_{A_{j}}\lc\bar{\xi}_{j}\rc{}}=\mathcal{W}_{\nabla_{\hspace{-0.055cm} j}}^{\log},\ h\vert_{\mathcal{C}_{A_{j}}\lc\bar{\xi}_{j}\rc{}}=h^{j},\ \restr{0.925}{\Enttau}{\mathcal{C}_{A_{j}}\lc\bar{\xi}_{j}\rc{}}=\Enttauj.
\end{align}

\noindent We suppress $j\in\mathbb{N}$ upon restriction as per Equation \ref{EQ.NTN.L2W_EVI_Equivalence_1}.
\end{ntn}

\begin{prp}\label{PRP.L2W_EVI_Equivalence}
Let $\xi\in\SII(A)$ be a finitely supported fixed state. Let $\mathcal{C}\subset (\SII(A),\mathcal{W}_{\nabla}^{\log})$ be finitely supported with fixed part $\xi$ s.t.~$\mathcal{C}\cap\dom\Enttau\neq\emptyset$. We have

\begin{itemize}
\item[1)] the metric-functional system $\big(\mathcal{C}\cap\dom\Enttau,\mathcal{W}_{\nabla}^{\log},\Enttau\big)$ equipped with continuous semigroup $h:[0,\infty)\times\mathcal{C}\cap\dom\Enttau\longrightarrow\mathcal{C}\cap\dom\Enttau$, \phantom{\big)}

\item[2)] the metric-functional system $\big(\mathcal{C}_{A_{j}}\lc\bar{\xi}_{j}\rc{},\mathcal{W}_{\nabla}^{\log},\Enttau\big)$ equipped with continuous semigroup $h:[0,\infty)\times\mathcal{C}_{A_{j}}\lc\bar{\xi}_{j}\rc\longrightarrow\mathcal{C}_{A_{j}}\lc\bar{\xi}_{j}\rc$ for a.e.~$j\in\mathbb{N}$, \phantom{\big)}

\item[3)] complete sublevels of $\Enttau:\mathcal{C}\cap\dom\Enttau\longrightarrow (-\infty,\infty)$ in $\mathcal{W}_{\nabla}^{\log}$-topology. \phantom{\big)}
\end{itemize}
\end{prp}
\begin{proof}
Using $\mathcal{C}\cap\dom\Enttau\neq\emptyset$, $3)$ in Corollary \ref{COR.QOT_Distance_AC_II} and $4)$ in Theorem \ref{THM.QOT_Distance_AC_FS} show we have metric-functional system $\lc\mathcal{C}\cap\dom\Enttau,\mathcal{W}_{\nabla}^{\log},\Enttau\rc$. Then $3)$ in Theorem \ref{THM.L2W_Log_Mean_NCDS} implies we obtain continuous semigroup $h:[0,\infty)\times\mathcal{C}\cap\dom\Enttau\longrightarrow\mathcal{C}\cap\dom\Enttau$ by considering $\restr{0.925}{h_{t}}{\mathcal{C}\cap\dom\Enttau}$ for all $t\geq 0$. Get $1)$. We see $2)$ follows since $\mathcal{C}_{A_{j}}\lc\bar{\xi}_{j}\rc\subset\dom\Enttau$ by Corollary \ref{COR.Rel_Ent_AF_Cstar_Trace} if $\xi_{j}\neq 0$. Using $3)$ in Corollary \ref{COR.QOT_Distance_AC_II}, l.s.c.~of $\Enttau$ implies $3)$.
\end{proof}

The conditions in Definition \ref{DFN.L2W_EVI_Equivalence} are subdivided in order as follows. First, three global conditions. Secondly, three local conditions with each one being the analogue of the respective global condition sharing its numbering. Thirdly, our Hessian lower bound condition. We ensure all such conditions are well-defined. For this, we collect results in our case concerning minimising geodesics, the coarse graining process and $\EVI_{\lambda}$-gradient flows of quantum relative entropy. We use Notation \ref{NTN.L2W_EVI_Equivalence}.\par
We collect results. Let $\xi\in\SII(A)$ be a finitely supported fixed state. Lemma \ref{LEM.AF_Support_Projection_Majorant_Uniform} shows $\xi$ has integrable support. Let $\mathcal{C}\subset (\SII(A),\mathcal{W}_{\nabla}^{\log})$ be finitely supported with fixed part $\xi$ s.t.~$\mathcal{C}\cap\dom\Enttau\neq\emptyset$. We have metric-functional systems equipped with heat flow as continuous semigroups as per $1)$ and $2)$ in Proposition \ref{PRP.L2W_EVI_Equivalence}. They coincide if $A$ and $B$ are finite-dimensional. Diagram \ref{EQ.SSEC.QOT_QIT_Encoding_16} for $K=\dom\Enttau$ shows we restrict, up to rescaling as per $1)$ in Definition \ref{DFN.AF_Cstar_Trace_Dualisation_Paths}, to

\begin{align}\label{EQ.SSEC.L2W_EVI_Equivalence_2}
\resj\lc\mathcal{C}\cap\dom\Enttau\rc{}=\mathcal{C}_{A_{j}}\lc\bar{\xi}_{j}\rc{},\ \resj\circ\hspace{0.0275cm} h=h\vert_{A_{j}^{*}}=h^{j}
\end{align}

\noindent for a.e.~$j\in\mathbb{N}$. Remark \ref{REM.AF_Cstar_Trace_Dualisation_Admissible_Paths} ensures we assume such rescaling here without loss of generality. Following $1)$ in Definition \ref{DFN.Metric_Functional_EVI_CNV}, $h:[0,\infty)\times\mathcal{C}\cap\dom\Enttau\longrightarrow\mathcal{C}\cap\dom\Enttau$ is $\EVI_{\lambda}$-gradient flow of $\Enttau$ in $\mathcal{C}\cap\dom\Enttau$ if for all $\mu,\eta\in\mathcal{C}\cap\dom\Enttau$, we have

\begin{align}\label{EQ.SSEC.L2W_EVI_Equivalence_3}
\frac{1}{2}\frac{d^{+}}{dt}\mathcal{W}_{\nabla}^{\log}\lc{}h_{t}(\mu),\eta\rc^{2}+\frac{\lambda}{2}\mathcal{W}_{\nabla}^{\log}\lc{}h_{t}(\mu),\eta\rc^{2}\leq\Ent(\eta,\tau)-\Ent\lc{}h_{t}(\mu),\tau\rc{}
\end{align}

\noindent for all $t\geq 0$. Remark \ref{REM.Metric_Functional_EVI_CNV} gives the following equivalent integral characterisation. For all $\mu,\eta\in\mathcal{C}\cap\dom\Enttau$, we have

\begin{align}\label{EQ.SSEC.L2W_EVI_Equivalence_4}
\frac{e^{\lambda\lc{}t-s\rc{}}}{2}\mathcal{W}_{\nabla}^{\log}\lc{}h_{t}(\mu),\eta\rc^{2}-\frac{1}{2}\mathcal{W}_{\nabla}^{\log}\lc{}h_{s}(\mu),\eta\rc^{2}\leq\int_{0}^{t-s}e^{\lambda r}dr\cdot \bigg(\hspace{-0.028975cm} \Ent(\eta,\tau)-\Ent\lc{}h_{t}(\mu),\tau\rc\bigg)
\end{align}

\noindent for all $0<s<t<\infty$. Note $4.2)$ in Theorem \ref{THM.QOT_Distance} ensures minimising geodesics and distance minimisers are identical. We assume $2.1)$ in Definition \ref{DFN.L2W_EVI_Equivalence}. Following $2)$ in Definition \ref{DFN.Metric_Functional_EVI_CNV}, $\Enttau$ is $\lambda$-convex if for all $\mu^{0},\mu^{1}\in\mathcal{C}\cap\dom\Enttau$ and $(\mu,w)\in\Geo\lc\mu^{0},\mu^{1}\rc$ s.t.~$\mu(t)\in\dom\Enttau$ for all $t\geq 0$, we have

\begin{align}\label{EQ.SSEC.L2W_EVI_Equivalence_5}
\Ent\lc\mu(t),\tau\rc\leq \lc{}1-t\rc\Ent\lc\mu^{0},\tau\rc{}+t\Ent\lc\mu^{1},\tau\rc{}-\frac{\lambda}{2}t\lc{}1-t\rc\mathcal{W}_{\nabla}^{\log}\lc\mu^{0},\mu^{1}\rc^{2}
\end{align}

\noindent for all $t\in [0,1]$. Following Definition \ref{DFN.Metric_Functional_EVI_CNV} and Remark \ref{REM.Metric_Functional_EVI_CNV}, let $\EVI_{\lambda}$, $\EVI_{\lambda}^{\int}$, resp.~$\CNV_{\lambda}$ reference the above equations accordingly. Note $\textrm{G}.3\rc$ in Definition \ref{DFN.L2W_EVI_Equivalence}, i.e.~$\BE_{\lambda}$, uses both Notation \ref{NTN.Wstar_Derivation_QG_HSG_Regularity} and $1.1)$ in Corollary \ref{COR.Wstar_Derivation_QG_HSG_Regularity}. Equation \ref{EQ.SSEC.L2W_EVI_Equivalence_2} shows we restrict in each step of the coarse graining process by replacing $\mathcal{C}\cap\dom\Enttau$ with $\mathcal{C}_{A_{j}}\lc\bar{\xi}_{j}\rc$. We thereby obtain local forms from global ones. For $\HI\rc$ in Definition \ref{DFN.L2W_EVI_Equivalence}, i.e.~$\HI_{\lambda}$, there exists no local form. Referenced equations do not use subscripts upon restriction.\par


\pagebreak


\begin{dfn}\label{DFN.L2W_EVI_Equivalence}
Let $\lambda\in\mathbb{R}$.

\begin{itemize}
\item[G.1)] We say that $\Enttau$ satisfies $\EVI_{\lambda}$ if for all finitely supported $\mathcal{C}\subset \big(\hspace{-0.03875cm} \SII(A),\mathcal{W}_{\nabla}^{\log}\big)$ s.t.~$\mathcal{C}\cap\dom\Enttau\neq\emptyset$, $h:[0,\infty)\times\mathcal{C}\cap\dom\Enttau\longrightarrow\mathcal{C}\cap\dom\Enttau$ is $\EVI_{\lambda}$-gradient flow of $\Enttau$ in $\mathcal{C}\cap\dom\Enttau$.

\item[G.2)] We say that $\Enttau$ satisfies $\lambda$-convexity if for all finitely supported $\mathcal{C}\subset \big(\hspace{-0.03875cm} \SII(A),\mathcal{W}_{\nabla}^{\log}\big)$ s.t.~$\mathcal{C}\cap\dom\Enttau\neq\emptyset$, we have

\begin{itemize}
\item[2.1)] $\big(\mathcal{C}\cap\dom\Enttau,\mathcal{W}_{\nabla}^{\log}\big)$ is a complete geodesic length-metric space, \phantom{\big)}

\item[2.2)] $\Enttau$ is $\lambda$-convex in $\mathcal{C}\cap\dom\Enttau$. \phantom{\big)}
\end{itemize}

\begin{reapply}
\end{reapply}

\item[G.3)] We say that $h$ satisfies $\BE_{\lambda}$ if for all finitely supported fixed states $\xi\in\SII(A)$ and $\mathcal{C}\subset \big(\hspace{-0.03875cm} \SII(A),\mathcal{W}_{\nabla}^{\log}\big)$ with fixed part $\xi$, we have

\begin{align}\label{EQ.DFN.L2W_EVI_Equivalence_1}
\dblv{}\mathcal{M}_{\sharp\mu}^{\frac{1}{2}}\nabla h_{t}(u)\dblv_{\omega}^{2}\leq e^{-2\lambda t}\dblv{}\mathcal{M}_{h_{t}(\sharp\mu)}^{\frac{1}{2}}\nabla u\dblv_{\omega}^{2} \tag{$\BE_{\lambda}$}
\end{align}

\begin{reapply}
\end{reapply}

\noindent for all $\mu\in\mathcal{C}\cap L^{2,\infty}(A_{\xi},\tau)^{\flat}$, $u\in\dom\nabla_{\hspace{-0.055cm} \xi}$ and $t\geq 0$.

\item[L.1)] We say that $\Enttau$ satisfies $\EVI_{\lambda}$ locally if for all finitely supported fixed states $\xi\in\SII(A)$, $h:[0,\infty)\times\mathcal{C}_{A_{j}}\lc\bar{\xi}_{j}\rc\longrightarrow\mathcal{C}_{A_{j}}\lc\bar{\xi}_{j}\rc$ is $\EVI_{\lambda}$-gradient flow of $\Enttau$ in $\mathcal{C}_{A_{j}}\lc\bar{\xi}_{j}\rc$ for a.e.~$j\in\mathbb{N}$.

\item[L.2)] We say that $\Enttau$ satisfies $\lambda$-convexity locally if for all finitely supported fixed states $\xi\in\SII(A)$, $\Enttau$ is $\lambda$-convex in $\mathcal{C}_{A_{j}}\lc\bar{\xi}_{j}\rc$ for a.e.~$j\in\mathbb{N}$.

\item[L.3)] We say that $h$ satisfies $\BE_{\lambda}$ locally if for all finitely supported fixed states $\xi\in\SII(A)$ and a.e.~$j\in\mathbb{N}$ in each case, we have

\begin{align}\label{EQ.DFN.L2W_EVI_Equivalence_2}
\dblv{}\mathcal{M}_{\sharp\mu}^{\frac{1}{2}}\nabla h_{t}(u)\dblv_{\omega}^{2}\leq e^{-2\lambda t}\dblv{}\mathcal{M}_{h_{t}(\sharp\mu)}^{\frac{1}{2}}\nabla u\dblv_{\omega}^{2} \tag{$\BE_{\lambda}^{\textrm{loc}}$}
\end{align}

\begin{reapply}
\end{reapply}

\noindent for all $\mu\in\mathcal{C}_{A_{j}}\lc\bar{\xi}_{j}\rc$, $u\in A_{j,\bar{\xi}_{j}}$ and $t\geq 0$. 

\item[H)] We say that $\Hess\Enttau$ has lower bound $\lambda$ if for all for all finitely supported fixed states $\xi\in\SII(A)$ and a.e.~$j\in\mathbb{N}$ in each case, we have

\begin{align}\label{EQ.DFN.L2W_EVI_Equivalence_3}
\Hess\hspace{-0.1525cm} \phantom{.}_{\mu}\Enttau(\eta)\geq\lambda g_{\mu}^{\bar{\xi}_{j}}(\eta,\eta) \tag{$\HI_{\lambda}$}
\end{align}

\begin{reapply}
\end{reapply}

\noindent for all $\mu\in\vartheta\lc\bar{\xi}_{j}\rc$ and $\eta\in I\big(\Delta_{\bar{\xi}_{j}}\big)^{\flat}$. We further call $\lambda$ a lower bound of the Hessian of quantum relative entropy and write $\Hess\Enttau\geq\lambda$.
\end{itemize}
\end{dfn}

Following results in Section \ref{SEC.QOT_DT}, Section \ref{SEC.QOT_AC} and Section \ref{SEC.L2W_Rel_Ent}, our discussion here and in Subsection \ref{SSEC.L2W_EVI_Ric} additionally uses below equations arising from applying the coarse graining process to quantum objects in the AF-$C^{*}$-setting. As such, they use compatibility with compression and finite-dimensional approximation. Let $\xi\in\SII(A)$ and $\mathcal{C}\subset (\SII(A),\mathcal{W}_{\nabla}^{\log})$ as above. We use $2.2)$ in Proposition \ref{PRP.Wstar_Derivation_QG_HSG_II}. The latter immediately reduces to Equation \ref{EQ.REM.Wstar_Derivation_QG_HSG_L2_1} in the square integrable case. For all $\mu^{0},\mu^{1}\in\SII(A)$, note $3)$ in Theorem \ref{THM.QOT_Distance} shows we have

\begin{align}\label{EQ.SSEC.L2W_EVI_Equivalence_6}
\mathcal{W}_{\nabla}^{\log}\lc{}h_{t}\lc\mu^{0}\rc{},h_{s}\lc\mu^{1}\rc\rc{}=\lim_{j\in\mathbb{N}}\hspace{0.025cm} \mathcal{W}_{\nabla}^{\log}\big(h_{t}\big(\bar{\mu}_{j}^{0}\big),h_{s}\big(\bar{\mu}_{j}^{1}\big)\big)
\end{align}

\noindent for all $t,s\geq 0$. For all $u\in L^{2}(A_{\xi},\tau)$, $2)$ and $4.1)$ in Corollary \ref{COR.Wstar_Derivation_QG_HSG_Regularity} together show we have $u\in\dom\nabla_{\hspace{-0.055cm} \xi}$ if and only if limits

\begin{align}\label{EQ.SSEC.L2W_EVI_Equivalence_7}
h_{t}(u)=\|.\|_{\nabla}\textrm{-}\lim_{j\in\mathbb{N}}\hspace{0.025cm} \pi_{\supp\xi_{j}}\lc{}h_{t}(u_{j})\rc{}=\|.\|_{\nabla}\textrm{-}\lim_{j\in\mathbb{N}}\hspace{0.025cm} \supp\xi_{j}h_{t}(u_{j})\supp\xi_{j}
\end{align}

\noindent exists for all $t\geq 0$. For all $\mu\in\mathcal{C}$, we see $3.2)$ in Theorem \ref{THM.QOT_Distance_AC_FS} shows we have

\begin{align}\label{EQ.SSEC.L2W_EVI_Equivalence_8}
\Ent\lc{}h_{t}(\mu),\tau\rc{}=\lim_{j\in\mathbb{N}}\hspace{0.025cm} \Ent\lc{}h_{t}\lc\mu_{j}\rc{},\tau\rc{}=\lim_{j\in\mathbb{N}}\hspace{0.025cm} \Ent\lc{}h_{t}\lc\bar{\mu}_{j}\rc{},\tau\rc{}
\end{align}

\noindent for all $t\in [0,\infty]$. Compare Equation \ref{EQ.SSEC.L2W_EVI_Equivalence_8} to Equation \ref{EQ.SSEC.L2W_Rel_Ent_AF_3}.
    
\begin{lem}\label{LEM.L2W_EVI_Equivalence}
Assume $A$ and $B$ are finite-dimensional. Let $\xi\in\SII(A)$ be a fixed state. For all $\lambda\in\mathbb{R}$, the following are equivalent:

\begin{itemize}
\item[1)] $h:[0,\infty)\times\mathcal{C}_{A}(\xi)\longrightarrow\mathcal{C}_{A}(\xi)$ is $\EVI_{\lambda}$-gradient flow of $\Enttau$ in $\mathcal{C}_{A}(\xi)$,

\item[2)] $\Enttau$ is $\lambda$-convex in $\mathcal{C}_{A}(\xi)$,

\item[3)] $\Enttau$ satisfies $\BE_{\lambda}$ for all $\mu\in\mathcal{C}_{A}(\xi)$, $u\in A_{\xi}$ and $t\geq 0$,

\item[4)] $\Enttau$ satisfies $\HI_{\lambda}$ for all $\mu\in\vartheta(\xi)$ and $\eta\in I(\Delta_{\xi})^{\flat}$.
\end{itemize}
\end{lem}
\begin{proof}
Let $\lambda\in\mathbb{R}$. We show $4)$ implies $1)$, then $1)$ implies $2)$, and finally $2)$ implies $4)$. We further show equivalence of $3)$ and $4)$. We thereby show our claim. Theorem \ref{THM.L2W_Log_Mean_NCDS_Fin_Hessian} lets us apply Theorem 2.2 in \cite{ART.Dan_Sav.2008.Classical_OT_GradFlow_DisConvex} to show $4)$ implies $1)$. We are motivated by analogous arguments in the proof of Theorem 4.5 in \cite{ART.Erb_Maa.2012.Discrete_OT_Ricci_Bounds}. However, Theorem \ref{THM.L2W_Log_Mean_NCDS_Fin_Hessian} replaces the essential steps in \cite{ART.Erb_Maa.2012.Discrete_OT_Ricci_Bounds} necessary to apply Theorem 2.2 in \cite{ART.Dan_Sav.2008.Classical_OT_GradFlow_DisConvex} here. Finally, further note Theorem \ref{THM.L2W_Log_Mean_NCDS_Fin_Hessian} lets us apply a standard semigroup interpolation argument as in the proof of Theorem 10.4 in \cite{ART.Car_Maa.2020.Quantum_OT_III} to show $3)$ implies $4)$.\par


\pagebreak


We show $4)$ implies $1)$. Using Corollary \ref{COR.QOT_Distance_AC_L2}, we know Theorem 3.3 in \cite{ART.Dan_Sav.2008.Classical_OT_GradFlow_DisConvex} implies $1)$ at once if $h:[0,\infty)\times\vartheta(\xi)\longrightarrow\vartheta(\xi)$ is $\EVI_{\lambda}$-gradient flow of $\Enttau$ in $\vartheta(\xi)$. For all smooth $\mu:[0,1]\longrightarrow\vartheta(\xi)$, set $\eta(t,s):=h_{ts}\lc\mu(t)\rc$ for all $t,s\geq 0$. Using the latter, Theorem 2.2 in \cite{ART.Dan_Sav.2008.Classical_OT_GradFlow_DisConvex} further shows $h:[0,\infty)\times\vartheta(\xi)\longrightarrow\vartheta(\xi)$ is $\EVI_{\lambda}$-gradient flow of $\Enttau$ in $\vartheta(\xi)$ if for all smooth $\mu:[0,1]\longrightarrow\vartheta(\xi)$, we have

\begin{align}\label{EQ.LEM.L2W_EVI_Equivalence_1}
\frac{1}{2}\frac{\partial}{\partial s}g_{\eta(t,s)}^{\xi}\lc\frac{\partial}{\partial t}\eta(t,s),\frac{\partial}{\partial t}\eta(t,s)\rc{}+\frac{\partial}{\partial t}\Enttau\lc\eta(t,s)\rc{} \leq -t\lambda g_{\eta(t,s)}^{\xi}\lc\frac{\partial}{\partial t}\eta(t,s),\frac{\partial}{\partial t}\eta(t,s)\rc{}
\end{align}

\noindent for all $t,s\geq 0$. Assume $4)$. Set $\varphi(t,s):=ts$ for all $t,s\geq 0$. Thus $\eta(t,s)=h_{\varphi(t,s)}\lc\mu(t)\rc$ in each case, hence $1)$ in Theorem \ref{THM.L2W_Log_Mean_NCDS_Fin_Hessian} yields

\begin{align}\label{EQ.LEM.L2W_EVI_Equivalence_2}
\frac{1}{2}\frac{\partial}{\partial s}g_{\eta(t,s)}^{\xi}\lc\frac{\partial}{\partial t}\eta(t,s),\frac{\partial}{\partial t}\eta(t,s)\rc{}+\frac{\partial}{\partial t}\Enttau\lc\eta(t,s)\rc{} & = -t\Hess\hspace{-0.1525cm} \phantom{.}_{\eta(t,s)}\Enttau\lc\frac{\partial}{\partial t}\eta(t,s)\rc{}
\end{align}

\noindent for all $t,s>0$. We extend to $t,s\geq 0$ by continuity. We apply H$_{\lambda}$ to the right-hand side of Equation \ref{EQ.LEM.L2W_EVI_Equivalence_2} and obtain Equation \ref{EQ.LEM.L2W_EVI_Equivalence_1}. Altogether, $4)$ implies $1)$.\par
Note $3)$ in Proposition \ref{PRP.Metric_Functional_EVI_CNV} shows $1)$ implies $2)$ since we know $\SII(A)\subset\dom\Enttau$ by finite-dimensionality. We show $2)$ implies $4)$. Assume $2)$. Let $\mu:[0,1]\longrightarrow\vartheta(\xi)$ be a minimising geodesic. Set $\mu:=\mu(0)$ and $\eta:=\dot{\mu}(0)$. Equation \ref{EQ.PRP.L2W_Log_Mean_NCDS_Fin_Hessian_4} states

\begin{align}\label{EQ.LEM.L2W_EVI_Equivalence_3}
\Hess\hspace{-0.1525cm} \phantom{.}_{\mu}\Enttau(\eta)=\restr{0.925}{\frac{d^{2}}{dr^{2}}}{r=0}\Enttau\lc\mu(r)\rc{}. 
\end{align}

We write out both differential quotients on the right-hand side of Equation \ref{EQ.LEM.L2W_EVI_Equivalence_3}. The latter equation therefore shows $\Hess\hspace{-0.1525cm} \phantom{.}_{\mu}\Enttau(\eta)$ equals

\begin{align}\label{EQ.LEM.L2W_EVI_Equivalence_4}
\lim_{t\downarrow 0}\hspace{0.025cm} \lim_{s\downarrow 0}\hspace{0.025cm} t^{-1}s^{-1}\bigg(\hspace{-0.028975cm} \Ent\lc\mu\lc{}t+s\rc{},\tau\rc{}-\Ent\lc\mu(t),\tau\rc{}-\Ent\lc\mu(s),\tau\rc{}+\Ent(\mu,\tau)\bigg).
\end{align}

\noindent Equation \ref{EQ.LEM.L2W_EVI_Equivalence_9} below lets us estimate parenthesis terms in Equation \ref{EQ.LEM.L2W_EVI_Equivalence_4}. Using the latter, we directly calculate lower Hessian bounds. For all $t,s\in (0,1)$ s.t.~$t+s<1$, set $\rho(r):=\mu\lc\lc{}t+s\rc{}r\rc$ and $\nu(r):=\mu\lc\lc{}t+s\rc\lc{}1-r\rc\rc$ for all $r\in[0,1]$. Segments of minimising geodesics reparametrised to constant speed on the unit interval as per Remark \ref{REM.Energy_Functional_Reparametrisation} are minimising geodesics. We obtain $\rho\in\Geo\lc\mu,\mu\lc{}t+s\rc\rc$ and $\nu\in\Geo\lc\mu\lc{}t+s\rc{},\mu\rc$ in each case, where we suppress canonical vector fields along minimising geodesics as per $1)$ in Proposition \ref{PRP.RM_III} in our notation here. We estimate by applying $\CNV_{\lambda}$ to the latter. We require additional considerations.\par


\pagebreak


Set $h(t,s):=t\lc{}t+s\rc^{-1}$ for all $t,s>0$. Using $h(t,s)=1-h(s,t)$, get $\mu(t)=\rho\lc{}h(t,s)\rc$ and $\mu(s)=\nu\lc{}1-h(s,t)\rc{}=\nu\lc{}h(t,s)\rc$ in each case. Let $t,s\in (0,1)$ s.t.~$t+s<1$. Minimising geodesics have $t$-a.e.~constant speed by $1)$ in Proposition \ref{PRP.QOT_Distance_Geodesics}. Since we have $\lc{}t+s\rc\eta=\dot{\rho}(0)$ by construction, symmetry of distances and constant speed of minimising geodesics let us calculate

\begin{align}\label{EQ.LEM.L2W_EVI_Equivalence_5}
\mathcal{W}_{\nabla}^{\log}\lc\mu\lc{}t+s\rc{},\mu\rc^{2}=\mathcal{W}_{\nabla}^{\log}\lc\mu,\mu\lc{}t+s\rc\rc^{2}=\lc{}t+s\rc^{2}\cdot g_{\mu}^{\xi}(\eta,\eta).
\end{align}

For all $t,s>0$ s.t.~$t+s<1$, $\CNV_{\lambda}$ and Equation \ref{EQ.LEM.L2W_EVI_Equivalence_5} let us calculate the following two estimates. First, $\CNV_{\lambda}$ to $\mu(t)=\rho\lc{}h(t,s)\rc$ in order to estimate

\begin{align}\label{EQ.LEM.L2W_EVI_Equivalence_6}
\Ent\lc\mu(t),\tau\rc\leq \big(1-h(t,s)\big)\cdot \Ent(\mu,\tau)+h(t,s)\cdot \Ent\lc\mu\lc{}t+s\rc{},\tau\rc{}-\frac{\lambda}{2}ts\cdot g_{\mu}^{\xi}(\eta,\eta).
\end{align}

\noindent Secondly, we apply $\CNV_{\lambda}$ to $\mu(s)=\nu\lc{}h(t,s)\rc$ in order to estimate

\begin{align}\label{EQ.LEM.L2W_EVI_Equivalence_7}
\Ent\lc\mu(s),\tau\rc\leq \big(1-h(t,s)\big)\cdot \Ent\lc\mu\lc{}t+s\rc{},\tau\rc{}+h(t,s)\cdot \Ent(\mu,\tau)-\frac{\lambda}{2}ts\cdot g_{\mu}^{\xi}(\eta,\eta).
\end{align}

\noindent We moreover add Equation \ref{EQ.LEM.L2W_EVI_Equivalence_6} and Equation \ref{EQ.LEM.L2W_EVI_Equivalence_7} to obtain

\begin{align}\label{EQ.LEM.L2W_EVI_Equivalence_8}
\Ent\lc\mu(t),\tau\rc{}+\Ent\lc\mu(s),\tau\rc\leq\Ent(\mu,\tau)+\Ent\lc\mu\lc{}t+s\rc{},\tau\rc{}-\lambda ts\cdot g_{\mu}^{\xi}(\eta,\eta)
\end{align}

\noindent in each case.\par
For all $t,s\in (0,1)$ s.t.~$t+s<1$, Equation \ref{EQ.LEM.L2W_EVI_Equivalence_8} implies

\begin{align}\label{EQ.LEM.L2W_EVI_Equivalence_9}
\lambda ts\cdot g_{\mu}^{\xi}(\eta,\eta)\leq\Ent\lc\mu\lc{}t+s\rc{},\tau\rc{}-\Ent\lc\mu(t),\tau\rc{}-\Ent\lc\mu(s),\tau\rc{}+\Ent(\mu,\tau)
\end{align}

\noindent by rearranging terms accordingly. Assuming $t+s<1$ in Equation \ref{EQ.LEM.L2W_EVI_Equivalence_4}, which we may do since we consider a double limit, Equation \ref{EQ.LEM.L2W_EVI_Equivalence_9} lets us estimate parenthesis terms in Equation \ref{EQ.LEM.L2W_EVI_Equivalence_4}. We therefore calculate

\begin{align*}
\Hess\hspace{-0.1525cm} \phantom{.}_{\mu}\Enttau(\eta) & = \restr{0.925}{\frac{d^{2}}{dr^{2}}}{r=0}\Enttau\lc\mu(r)\rc \phantom{\Bigg)} \\
& =\lim_{t\downarrow 0}\hspace{0.025cm} \lim_{s\downarrow 0}\hspace{0.025cm} t^{-1}s^{-1}\bigg(\hspace{-0.028975cm} \Ent\lc\mu\lc{}t+s\rc{},\tau\rc{}-\Ent\lc\mu(t),\tau\rc{}-\Ent\lc\mu(s),\tau\rc{}+\Ent(\mu,\tau)\bigg) \phantom{\Bigg)} \\
& \geq\lambda g_{\mu}^{\xi}(\eta,\eta). \phantom{\Bigg)}
\end{align*}

\noindent The above calculation shows $2)$ implies $4)$. Corollary \ref{COR.RM} ensures we have sufficient minimising geodesics in $\vartheta(\xi)$. Altogether, we obtain a chain of implications as claimed.\par


\pagebreak


We show equivalence of $3)$ and $4)$. Assume $3)$. Let $\mu\in\vartheta(\xi)$ and $x\in I(\Delta_{\xi})$. Set

\begin{align}\label{EQ.LEM.L2W_EVI_Equivalence_10}
l(t):=e^{2\lambda t}\dblv{}\mathcal{M}_{\sharp\mu}^{\frac{1}{2}}\nabla h_{t}(x)\dblv_{\omega}^{2},\ r(t):=\dblv{}\mathcal{M}_{h_{t}(\sharp\mu)}^{\frac{1}{2}}\nabla x\dblv_{\omega}^{2}
\end{align}

\noindent for all $t\geq 0$. Applying $\BE_{\lambda}$ to Equation \ref{EQ.LEM.L2W_EVI_Equivalence_10} shows $r(t)\geq l(t)$ for all $t\geq 0$ and therefore $r'(0)\geq l'(0)$ as well. We directly verify

\begin{align}\label{EQ.LEM.L2W_EVI_Equivalence_11}
l'(0)=-2g_{\mu}^{\xi}\lc\mathfrak{F}_{\mu}\lc\Delta x\rc^{\flat},\mathfrak{F}_{\mu}(x)^{\flat}\rc{}+2\lambda g_{\mu}^{\xi}\lc\mathfrak{F}_{\mu}(x)^{\flat},\mathfrak{F}_{\mu}(x)^{\flat}\rc{}.
\end{align}

Applying Proposition \ref{PRP.Differential_Inverse} to $\mathfrak{F}=\lc\mathfrak{F}^{-1}\rc^{-1}$ and further using Lemma \ref{LEM.L2W_Log_Mean_NCDS_Fin_EL}, we thus argue as in the proof of $2)$ in Theorem \ref{THM.L2W_Log_Mean_NCDS_Fin_Hessian} in order to calculate

\begin{align}\label{EQ.LEM.L2W_EVI_Equivalence_12}
r'(0)=-\bigg\langle\Lambda_{\mu}^{*}\lc\sharp\Theta\big(\mu,\mathfrak{F}_{\mu}(x)^{\flat}\big),\sharp\Theta\big(\mu,\mathfrak{F}_{\mu}(x)^{\flat}\big)\rc{},\Delta\mu\bigg\rangle_{\tau}.
\end{align}

\noindent Proposition \ref{PRP.L2W_Log_Mean_NCDS_Fin_Hessian} shows

\begin{align*}
\Hess\hspace{-0.1525cm} \phantom{.}_{\mu}\Enttau\lc\mathfrak{F}_{\mu}(x)^{\flat}\rc{} =& -\bigg\langle\frac{1}{2}\Lambda_{\mu}^{*}\lc\sharp\Theta\big(\mu,\mathfrak{F}_{\mu}(x)^{\flat}\big),\sharp\Theta\big(\mu,\mathfrak{F}_{\mu}(x)^{\flat}\big)\rc{},\Delta\mu\bigg\rangle_{\tau} \phantom{\Bigg)} \\
& +g_{\mu}^{\xi}\lc\mathfrak{F}_{\mu}\lc\Delta x\rc^{\flat},\mathfrak{F}_{\mu}(x)^{\flat}\rc{}. \phantom{\Bigg)}
\end{align*}

\noindent Using $r'(0)\geq l'(0)$ and the above identity, Equation \ref{EQ.LEM.L2W_EVI_Equivalence_11} and Equation \ref{EQ.LEM.L2W_EVI_Equivalence_12} imply

\begin{align}\label{EQ.LEM.L2W_EVI_Equivalence_13}
\Hess\hspace{-0.1525cm} \phantom{.}_{\mu}\Enttau\lc\mathfrak{F}_{\mu}(x)^{\flat}\rc\geq\lambda g_{\mu}^{\xi}\lc\mathfrak{F}_{\mu}(x)^{\flat},\mathfrak{F}_{\mu}(x)^{\flat}\rc{}
\end{align}

\noindent by rearranging terms accordingly. Note $\mathfrak{F}_{\mu}$ in Equation \ref{EQ.LEM.L2W_EVI_Equivalence_13} is of no consequence by $1)$ in Proposition \ref{PRP.RM_I}. Thus Equation \ref{EQ.LEM.L2W_EVI_Equivalence_13} shows $\Enttau$ satisfies $\HI_{\lambda}$. Get $4)$.\par
Assume $4)$. It suffices to consider $\mu\in\vartheta(\xi)$ by Corollary \ref{COR.QOT_Distance_AC_L2}, as well as $x\in I(\Delta_{\xi})$ by $1)$ in Corollary \ref{COR.AF_NCD_FC_III} and symmetry of $\nabla$. Let $U:=\lset (t,s)\in (0,\infty)\times (0,\infty)\ \vset\ t>s\rset$. Set $\varphi_{0}(t,s):=s$ and $\varphi_{1}(t,s):=t-s$, as well as

\begin{align}\label{EQ.LEM.L2W_EVI_Equivalence_14}
\eta(t,s):=h_{\varphi_{0}(t,s)}(\mu)=h_{s}(\mu),\ X(t,s):=h_{\varphi_{1}(t,s)}(x)=h_{t-s}(x)
\end{align}

\noindent for all $(t,s)\in U$. Thus $\frac{\partial}{\partial s}\varphi_{0}=-\frac{\partial}{\partial s}\varphi_{1}$, hence $2)$ in Theorem \ref{THM.L2W_Log_Mean_NCDS_Fin_Hessian} yields

\begin{align}\label{EQ.LEM.L2W_EVI_Equivalence_15}
\frac{1}{2}\frac{\partial}{\partial s}\dblv{}\mathcal{M}_{h_{s}(\sharp\mu)}^{\frac{1}{2}}\nabla h_{t-s}(u)\dblv_{\omega}^{2}=\Hess\hspace{-0.1525cm} \phantom{.}_{h_{s}(\mu)}\Enttau\lc\mathfrak{F}_{h_{s}(\mu)}\big(h_{t-s}(x)\big)^{\flat}\rc{}
\end{align}

\noindent for all $(t,s)\in U$. If $t>0$, then we extend to all $s\in [0,t]$ by continuity.\par


\pagebreak


For all $t>0$, set 

\begin{align}\label{EQ.LEM.L2W_EVI_Equivalence_16}
l(s):=e^{-2\lambda s}\cdot \dblv{}\mathcal{M}_{h_{s}(\sharp\mu)}^{\frac{1}{2}}\nabla h_{t-s}(x)\dblv_{\omega}^{2}
\end{align}

\noindent for all $s\in [0,t]$. Applying Equation \ref{EQ.LEM.L2W_EVI_Equivalence_15} to derivatives of terms in Equation \ref{EQ.LEM.L2W_EVI_Equivalence_16} lets us calculate

\begin{align*}
l'(s) & = 2e^{-2\lambda s}\cdot \lc\frac{1}{2}\frac{\partial}{\partial s}\dblv{}\mathcal{M}_{h_{s}(\sharp\mu)}^{\frac{1}{2}}\nabla h_{t-s}(x)\dblv_{\omega}^{2}-\lambda \dblv{}\mathcal{M}_{h_{s}(\sharp\mu)}^{\frac{1}{2}}\nabla h_{t-s}(x)\dblv_{\omega}^{2}\rc \phantom{\Bigg)} \\
& = 2e^{-2\lambda s}\cdot \lc\Hess\hspace{-0.1525cm} \phantom{.}_{h_{s}(\mu)}\Enttau\lc\mathfrak{F}_{h_{s}(\sharp\mu)}\big(h_{t-s}(x)\big)^{\flat}\rc{}-\lambda \dblv{}\mathcal{M}_{h_{s}(\sharp\mu)}^{\frac{1}{2}}\nabla h_{t-s}(x)\dblv_{\omega}^{2}\rc{} \phantom{\Bigg)}
\end{align*}

\noindent in each case. Further note $\mathfrak{F}_{h_{s}(\mu)}=\nabla^{*}\mathcal{M}_{h_{s}(\mu)}\nabla$ on $I(\Delta_{\xi})$. We obtain

\begin{align}\label{EQ.LEM.L2W_EVI_Equivalence_17}
\dblv{}\mathcal{M}_{h_{s}(\sharp\mu)}^{\frac{1}{2}}\nabla h_{t-s}(x)\dblv_{\omega}^{2}=g_{h_{s}(\mu)}^{\xi}\lc\mathfrak{F}_{h_{s}(\mu)}\big(h_{t-s}(x)\big)^{\flat},\mathfrak{F}_{h_{s}(\mu)}\big(h_{t-s}(x)\big)^{\flat}\rc{}
\end{align}

\noindent for all $s\in [0,t]$. Then applying Equation \ref{EQ.LEM.L2W_EVI_Equivalence_17} to its preceding calculation yields

\begin{align*}
l'(s)=2e^{-2\lambda s}\cdot \Bigg( & \Hess\hspace{-0.1525cm} \phantom{.}_{h_{s}(\mu)}\Enttau\lc\mathfrak{F}_{h_{s}(\mu)}h_{t-s}(x)\rc \phantom{\Bigg)} \\
&-\lambda g_{h_{s}(\mu)}^{\xi}\lc\mathfrak{F}_{h_{s}(\mu)}\big(h_{t-s}(x)\big)^{\flat},\mathfrak{F}_{h_{s}(\mu)}\big(h_{t-s}(x)\big)^{\flat}\rc\Bigg). \phantom{\Bigg)}
\end{align*}

If $t>0$, then the above calculation shows $\HI_{\lambda}$ implies $l'(s)\geq 0$ for all $s\in [0,t]$. For all $t\geq 0$, we therefore have $l(t)\geq l(0)$. Using the latter, Equation \ref{EQ.LEM.L2W_EVI_Equivalence_16} implies

\begin{align}\label{EQ.LEM.L2W_EVI_Equivalence_18}
\dblv{}\mathcal{M}_{\sharp\mu}^{\frac{1}{2}}\nabla h_{t}(x)\dblv_{\omega}^{2}\leq e^{-2\lambda t}\dblv{}\mathcal{M}_{h_{t}(\sharp\mu)}^{\frac{1}{2}}\nabla x\dblv_{\omega}^{2}
\end{align}

\noindent for all $t\geq 0$. Equation \ref{EQ.LEM.L2W_EVI_Equivalence_18} shows $\Enttau$ satisfies $\BE_{\lambda}$ at once. Get $3)$. Altogether, get equivalence of $3)$ and $4)$.
\end{proof}

\begin{thm}\label{THM.L2W_EVI_Equivalence}
Let $(\phi,\bpsi,\gamma,\nabla)$ be noncommutative differential structure for tracial AF-$C^{*}$-algebras $(A,\tau)$ and $(B,\omega)$ in the logarithmic mean setting. For all $\lambda\in\mathbb{R}$, the conditions in Definition \ref{DFN.L2W_EVI_Equivalence} are equivalent.
\end{thm}
\begin{proof}
Let $\lambda\in\mathbb{R}$. Note Equation \ref{EQ.SSEC.L2W_EVI_Equivalence_2} at once shows Lemma \ref{LEM.L2W_EVI_Equivalence} implies equivalence of $\textrm{L}.1\rc$, $\textrm{L}.2\rc$, $\textrm{L}.3\rc$ and $\HI\rc$. It suffices to show equivalence of $\textrm{G}.1\rc$ and $\textrm{L}.1\rc$, of $\textrm{G}.2\rc$ and $\textrm{L}.2\rc$, as well as of $\textrm{G}.3\rc$ and $\textrm{L}.3\rc$ each. We do so by passing from global to local properties and vice versa by means of the coarse graining process. We consider the following fixed but arbitrary. Let $\xi\in\SII(A)$ be a finitely supported fixed state. Let $\mathcal{C}\subset (\SII(A),\mathcal{W}_{\nabla}^{\log})$ be finitely supported with fixed part $\xi$ s.t.~$\mathcal{C}\cap\dom\Enttau\neq\emptyset$. We test all statements on the latter without loss of generality.\par
We show equivalence of $\textrm{G}.1\rc$ and $\textrm{L}.1\rc$. Assume $\textrm{G}.1\rc$. For a.e.~$j\in\mathbb{N}$, note $\mathcal{C}_{A}\lc\bar{\xi}_{j}\rc$ is finitely supported s.t.~$\bar{\xi}_{j}\in\mathcal{C}_{A}\lc\bar{\xi}_{j}\rc\cap\dom\Enttau\neq\emptyset$. If the latter is satisfied, then $\textrm{G}.1\rc$ implies $h:[0,\infty)\times\mathcal{C}_{A}\lc\bar{\xi}_{j}\rc\cap\dom\Enttau\longrightarrow\mathcal{C}_{A}\lc\bar{\xi}_{j}\rc\cap\dom\Enttau$ is $\EVI_{\lambda}$-gradient flow of $\Enttau$ in $\mathcal{C}_{A}\lc\bar{\xi}_{j}\rc$. Moreover, $2)$ in Theorem \ref{THM.QOT_Distance} yields isometric inclusion

\begin{align}\label{EQ.THM.L2W_EVI_Equivalence_1}
\big(\mathcal{C}_{A_{j}}\lc\bar{\xi}_{j}\rc{},\mathcal{W}_{\nabla}^{\log}\big)\subset\big(\mathcal{C}_{A}\lc\bar{\xi}_{j}\rc\cap\dom\Enttau,\mathcal{W}_{\nabla}^{\log}\big)
\end{align}

\noindent in each case. For a.e.~$j\in\mathbb{N}$, Equation \ref{EQ.THM.L2W_EVI_Equivalence_1} reduces $\EVI_{\lambda}$ as per Equation \ref{EQ.SSEC.L2W_EVI_Equivalence_3} from $\mathcal{C}_{A}\lc\bar{\xi}_{j}\rc$ to $\mathcal{C}_{A_{j}}\lc\bar{\xi}_{j}\rc$, i.e.~we see $h:[0,\infty)\times\mathcal{C}_{A_{j}}\lc\bar{\xi}_{j}\rc\longrightarrow\mathcal{C}_{A_{j}}\lc\bar{\xi}_{j}\rc$ is $\EVI_{\lambda}$-gradient flow of $\Enttau$ in $\mathcal{C}_{A_{j}}\lc\bar{\xi}_{j}\rc$ in each case. Get $\textrm{L}.1\rc$.\par
Assume $\textrm{L}.1\rc$. Let $\mu,\eta\in\mathcal{C}\cap\dom\Enttau$. For a.e.~$j\in\mathbb{N}$, note $\textrm{L}.1\rc$ shows $\EVI_{\lambda}^{\int}$ as per Equation \ref{EQ.SSEC.L2W_EVI_Equivalence_4} for $\bar{\mu}_{j},\bar{\eta}_{j}\in\mathcal{C}_{A_{j}}\lc\bar{\xi}_{j}\rc$. Using the latter, Equation \ref{EQ.SSEC.L2W_EVI_Equivalence_6} and Equation \ref{EQ.SSEC.L2W_EVI_Equivalence_8} let us estimate

\begin{align*}
& \frac{e^{\lambda\lc{}t-s\rc{}}}{2}\mathcal{W}_{\nabla}^{\log}\lc{}h_{t}(\mu),\eta\rc^{2}-\frac{1}{2}\mathcal{W}_{\nabla}^{\log}\lc{}h_{s}(\mu),\eta\rc^{2} \phantom{\vstretch{1.15}{\Bigg)}} \\
=& \lim_{j\in\mathbb{N}}\hspace{0.025cm} \frac{e^{\lambda\lc{}t-s\rc{}}}{2}\mathcal{W}_{\nabla}^{\log}\lc{}h_{t}\lc\bar{\mu}_{j}\rc{},\bar{\eta}_{j}\rc^{2}-\frac{1}{2}\mathcal{W}_{\nabla}^{\log}\lc{}h_{s}\lc\bar{\mu}_{j}\rc{},\bar{\eta}_{j}\rc^{2} \phantom{\vstretch{1.15}{\Bigg)}} \\
\leq& \lim_{j\in\mathbb{N}}\hspace{0.025cm} \int_{0}^{t-s}e^{\lambda r}dr\cdot \bigg(\hspace{-0.028975cm} \Ent\lc\bar{\eta}_{j},\tau\rc{}-\Ent\lc{}h_{t}\lc\bar{\mu}_{j}\rc{},\tau\rc\bigg) \phantom{\vstretch{1.15}{\Bigg)}} \\
=& \int_{0}^{t-s}e^{\lambda r}dr\cdot \bigg(\hspace{-0.028975cm} \Ent(\eta,\tau)-\Ent\lc{}h_{t}(\mu),\tau\rc\bigg) \phantom{\vstretch{1.15}{\Bigg)}}
\end{align*}

\noindent for all $0<s<t<\infty$. The above calculation readily lifts $\EVI_{\lambda}^{\int}$ as per Equation \ref{EQ.SSEC.L2W_EVI_Equivalence_4} from $\lset\mathcal{C}_{A_{j}}\lc\bar{\xi}_{j}\rc\rset_{j\in\mathbb{N}}$ to $\mathcal{C}\cap\dom\Enttau$, i.e.~we see $h:[0,\infty)\times\mathcal{C}\cap\dom\Enttau\longrightarrow\mathcal{C}\cap\dom\Enttau$ is $\EVI_{\lambda}$-gradient flow of $\Enttau$ in $\mathcal{C}\cap\dom\Enttau$ in each case. Get $\textrm{G}.1\rc$. Altogether, get equivalence of $\textrm{G}.1\rc$ and $\textrm{L}.1\rc$.\par
We show equivalence of $\textrm{G}.2\rc$ and $\textrm{L}.2\rc$. Assume $\textrm{G}.2\rc$. We then reduce from global to local property as above. For a.e.~$j\in\mathbb{N}$, we see $2.1)$ in Proposition \ref{PRP.QOT_Distance_Geodesics} shows

\begin{align}\label{EQ.THM.L2W_EVI_Equivalence_2}
\textrm{Geo}_{j}\big(\bar{\mu}_{j}^{0},\bar{\mu}_{j}^{1}\big)\subset\Geo\big(\bar{\mu}_{j}^{0},\bar{\mu}_{j}^{1}\big)
\end{align}

\noindent for all $\mu^{0},\mu^{1}\in\mathcal{C}_{A_{j}}\lc\bar{\xi}_{j}\rc$. For a.e.~$j\in\mathbb{N}$, Equation \ref{EQ.THM.L2W_EVI_Equivalence_1} and Equation \ref{EQ.THM.L2W_EVI_Equivalence_2} reduce $\CNV_{\lambda}$ as per Equation \ref{EQ.SSEC.L2W_EVI_Equivalence_5} from $\mathcal{C}_{A}\lc\bar{\xi}_{j}\rc$ to $\mathcal{C}_{A_{j}}\lc\bar{\xi}_{j}\rc$, i.e.~we see $\Enttau$ is $\lambda$-convex in $\mathcal{C}_{A_{j}}\lc\bar{\xi}_{j}\rc$ in each case. Get $\textrm{L}.2\rc$. Assume $\textrm{L}.2\rc$. We show $\textrm{G}.2\rc$ by using equivalence of $\textrm{G}.1\rc$ and $\textrm{L}.1\rc$ to apply $3)$ in Proposition \ref{PRP.Metric_Functional_EVI_CNV}. We show $2.1)$ in Definition \ref{DFN.L2W_EVI_Equivalence}. Let $\mu^{0},\mu^{1}\in\mathcal{C}\cap\dom\Enttau$. Since they are at finite distance, Theorem \ref{THM.QOT_Minimiser_Approximation} shows there exists $(\mu,w)\in\Geo\lc\mu^{0},\mu^{1}\rc$ approximated in finite dimensions by a sequence $\lc\mu^{j},w^{j}\rc_{j\geq m}\subset\Geo_{0}$. For a.e.~$j\in\mathbb{N}$, $\textrm{L}.2\rc$ shows $\CNV_{\lambda}$ as per Equation \ref{EQ.SSEC.L2W_EVI_Equivalence_5} for the minimising geodesic $\mu^{j}:[0,1]\longrightarrow\mathcal{C}_{A_{j}}\lc\bar{\xi}_{j}\rc$.\par


\pagebreak


Upon passing to a subsequence converging to $(\mu,w)$ in $\Admnullone$, we consider $\CNV_{\lambda}$ in each case and take limits in $j\in\mathbb{N}$ on both sides. Equation \ref{EQ.SSEC.L2W_EVI_Equivalence_6} and Equation \ref{EQ.SSEC.L2W_EVI_Equivalence_8} show they exist. We therefore have $C>0$ s.t.~

\begin{align}\label{EQ.THM.L2W_EVI_Equivalence_3}
\Ent\lc\mu^{j}(t),\tau\rc\leq C\cdot \max\left\{\Ent\big(\bar{\mu}_{j}^{0},\tau\big),\Ent\big(\bar{\mu}_{j}^{1},\tau\big)\right\}
\end{align}

\noindent for a.e.~$j\in\mathbb{N}$. Equation \ref{EQ.THM.L2W_EVI_Equivalence_3} shows Corollary \ref{COR.QOT_Distance_AC_FS} applies. The latter in turn shows $\mu(t)\in\mathcal{C}\cap\dom\Enttau$ for all $t\in [0,1]$. Ergo $2.1)$ as claimed. If $\textrm{G}.1\rc$ holds, then $\textrm{G}.2\rc$ follows by $3)$ in Proposition \ref{PRP.Metric_Functional_EVI_CNV}. Lemma \ref{LEM.L2W_EVI_Equivalence} shows $\textrm{L}.2\rc$ implies $\textrm{L}.1\rc$. The latter is equivalent to $\textrm{G}.1\rc$. Get $\textrm{G}.2\rc$. Altogether, get equivalence of $\textrm{G}.2\rc$ and $\textrm{L}.2\rc$.\par
We show equivalence of $\textrm{G}.3\rc$ and $\textrm{L}.3\rc$. Assume $\textrm{G}.3\rc$. We reduce from global to local property as above. For a.e.~$j\in\mathbb{N}$, Equation \ref{EQ.THM.L2W_EVI_Equivalence_1} reduces $\BE_{\lambda}$ from $\mathcal{C}_{A}\lc\bar{\xi}_{j}\rc$ to $\mathcal{C}_{A_{j}}\lc\bar{\xi}_{j}\rc$. Get $\textrm{L}.3\rc$. Assume $\textrm{L}.3\rc$. Using $3)$ in Proposition \ref{PRP.AF_Cstar_Trace_Dualisation_II}, $2.2)$ in Proposition \ref{PRP.Wstar_Derivation_QG_HSG_II}, which reduces to Equation \ref{EQ.REM.Wstar_Derivation_QG_HSG_L2_1} here, and Equation \ref{EQ.SSEC.L2W_EVI_Equivalence_7} show

\begin{align}\label{EQ.THM.L2W_EVI_Equivalence_4}
h_{t}\lc\sharp\mu\rc{}=\s\textrm{-}\lim_{j\in\mathbb{N}}\hspace{0.025cm} h_{t}\lc\sharp\bar{\mu}_{j}\rc{},\ h_{t}(u)=\|.\|_{\nabla}\textrm{-}\lim_{j\in\mathbb{N}}\hspace{0.025cm} \pi_{\supp\xi_{j}}\lc{}h_{t}(u_{j})\rc{}
\end{align}

\noindent for all $\mu\in\mathcal{C}\cap L^{2,\infty}(A_{\xi},\tau)^{\flat}$, $u\in\dom\nabla_{\hspace{-0.055cm} \xi}$ and $t\geq 0$. Using Lemma \ref{LEM.FC_SR}, for which we ensure necessary and suitable uniform boundedness by $2.1)$ in Proposition \ref{PRP.AF_Cstar_Trace_Dualisation_II}, the left-hand side of Equation \ref{EQ.THM.L2W_EVI_Equivalence_4} implies

\begin{align}\label{EQ.THM.L2W_EVI_Equivalence_5}
\mathcal{M}_{h_{t}(\sharp\mu)}^{\frac{1}{2}}=\s\textrm{-}\lim_{j\in\mathbb{N}}\hspace{0.025cm} \mathcal{M}_{h_{t}(\sharp\bar{\mu}_{j})}^{\frac{1}{2}}
\end{align}

\noindent for all $\mu\in\mathcal{C}\cap L^{2,\infty}(A_{\xi},\tau)^{\flat}$ and $t\geq 0$ \lc{}cf.~Remark \ref{REM.SR_Equivalence} and Remark \ref{REM.FC_SR}\rc{}.\par
Finally, we estimate. For a.e.~$j\in\mathbb{N}$, note $\textrm{L}.3\rc$ shows $\BE_{\lambda}$ for all $\mu\in\mathcal{C}_{A_{j}}\lc\bar{\xi}_{j}\rc$, $u_{j}\in A_{j,\bar{\xi}_{j}}$ and $t\geq 0$. Using the latter, the right-hand side of Equation \ref{EQ.THM.L2W_EVI_Equivalence_4} and Equation \ref{EQ.THM.L2W_EVI_Equivalence_5} let us estimate

\begin{align*}
\dblv{}\mathcal{M}_{\sharp\mu}^{\frac{1}{2}}\nabla h_{t}(u)\dblv_{\omega}^{2} & = \lim_{j\in\mathbb{N}}\hspace{0.025cm} \dblv{}\mathcal{M}_{\sharp\bar{\mu}_{j}}^{\frac{1}{2}}\nabla h_{t}(u_{j})\dblv_{\omega}^{2} \phantom{\Bigg)} \\
& \leq\lim_{j\in\mathbb{N}}\hspace{0.025cm} e^{-2\lambda t}\dblv{}\mathcal{M}_{h_{t}(\sharp\bar{\mu}_{j})}^{\frac{1}{2}}\nabla u_{j}\dblv_{\omega}^{2} \phantom{\Bigg)} \\
& = e^{-2\lambda t}\dblv{}\mathcal{M}_{h_{t}(\sharp\mu)}^{\frac{1}{2}}\nabla u\dblv_{\omega}^{2} \phantom{\Bigg)} 
\end{align*}

\noindent for all $\mu\in\mathcal{C}\cap L^{2,\infty}(A_{\xi},\tau)^{\flat}$, $u\in\dom\nabla_{\hspace{-0.055cm} \xi}$ and $t\geq 0$. The above calculation lifts $\BE_{\lambda}$ from $\lset\mathcal{C}_{A_{j}}\lc\bar{\xi}_{j}\rc\rset_{j\in\mathbb{N}}$ to $\mathcal{C}\cap\dom\Enttau$. Get $\textrm{G}.3\rc$. Altogether, get equivalence of $\textrm{G}.3\rc$ and $\textrm{L}.3\rc$.
\end{proof}


\pagebreak


\begin{cor}\label{COR.L2W_EVI_Equivalence}
Let $\Enttau$ satisfy $\EVI_{\lambda}$ for $\lambda\in\mathbb{R}$. Let $S:[0,\infty)\times\SII(A)\longrightarrow\SII(A)$ be a continuous semigroup s.t.~$S_{t}:\SII(A)\longrightarrow\SII(A)$ is $w^{*}$-continuous for all $t\geq 0$. If we know $S:[0,\infty)\times\mathcal{C}\cap\dom\Enttau\longrightarrow\mathcal{C}\cap\dom\Enttau$ is $\EVI_{\lambda}$-gradient flow of $\Enttau$ in $\mathcal{C}\cap\dom\Enttau$ for all finitely supported $\mathcal{C}\subset (\SII(A),\mathcal{W}_{\nabla}^{\log})$ s.t.~$\mathcal{C}\cap\dom\Enttau\neq\emptyset$, then $S=h$.
\end{cor}
\begin{proof}
Let $\lambda\in\mathbb{R}$ as per hypothesis. For all finitely supported fixed states $\xi\in\SII(A)$, we know $S:[0,\infty)\times\mathcal{C}_{A}\lc\bar{\xi}_{j}\rc\cap\dom\Enttau\longrightarrow\mathcal{C}_{A}\lc\bar{\xi}_{j}\rc\cap\dom\Enttau$ is $\EVI_{\lambda}$-gradient flow of $\Enttau$ in $\mathcal{C}_{A}\lc\bar{\xi}_{j}\rc\cap\dom\Enttau\neq\emptyset$ for a.e.~$j\in\mathbb{N}$. Uniqueness of $\EVI_{\lambda}$-gradient flows \cite{ART.Mur_Sav.2020.Classical_OT_EVI} implies $S_{t}(\mu)=h_{t}(\mu)$ for all $\mu\in\mathcal{C}_{A}\lc\bar{\xi}_{j}\rc$ and $t\geq 0$ in each case. Diagram \ref{EQ.SSEC.QOT_QIT_Encoding_16} for $K=\dom\Enttau$ furthermore shows 

\begin{align}\label{EQ.COR.L2W_EVI_Equivalence_1}
\SII(A)=\overline{\bigcup_{\xi\in\SII(A)}\bigcup_{j\in\mathbb{N}}\mathcal{C}_{A}\lc\bar{\xi}_{j}\rc{}}
\end{align}

\noindent in $w^{*}$-topology. However, each non-vanishing $\bar{\xi}_{j}\in\SII(A)$ is a finitely supported fixed state itself. Equation \ref{EQ.SSEC.L2W_EVI_Equivalence_2} therefore implies we may reduce to finitely supported $\xi\in\SII(A)$ in the first product on the right-hand side of Equation \ref{EQ.COR.L2W_EVI_Equivalence_1}. The latter therefore implies $S_{t}=h_{t}$ for all $t\geq 0$ by $w^{*}$-continuity.
\end{proof}


\subsection{Lower Ricci bounds}\label{SSEC.L2W_EVI_Ric}
        
We define lower Ricci bounds of quantum gradients using conditions in Definition \ref{DFN.L2W_EVI_Equivalence}. Theorem \ref{THM.L2W_EVI_Equivalence} ensures all such conditions are indeed equivalent. Lower Ricci bounds are given by $\lambda$-convexity of quantum information along minimising geodesics measured by quantum relative entropy. Their non-spatiality is further visible beyond the given description in terms of quantum information theory \cite{BK.Nie_Chu.2000.Quantum_Computation_Information} as follows. Assuming strictly positive lower Ricci bounds and finitely supported fixed part, Theorem \ref{THM.QOT_Distance_AC_Rel_Ent} classifies accessibility components of normal states with finite quantum relative entropy using fixed parts. Using the latter, we show strictly positive lower Ricci bounds determine energy-information trade-offs pa\-ra\-metrised by lower bounds on quantum noise.\par
Moreover, we extend remaining results in \cite{ART.Car_Maa.2014.Quantum_OT_I}\cite{ART.Car_Maa.2017.Quantum_OT_II}\cite{ART.Car_Maa.2020.Quantum_OT_III} as claimed. Theorem \ref{THM.L2W_Ric} gives sufficient conditions for lower Ricci bounds of direct sum quantum gradients. Apart from generalised discrete derivatives over finite sets, Theorem \ref{THM.L2W_Ric} applies to all fundamental example classes in Subsection \ref{SSEC.QOT_DT_BSP}. Theorem \ref{THM.L2W_Ric_FI} derives functional inequalities and their chain of implications. Note all terms correcting for non-ergodicity are given by quantum relative entropy evaluated on finitely supported fixed parts since conditioning is determined by the underlying metric geometry as restriction to finitely supported accessibility components.


\subsubsection*{Definition and energy-information trade-offs from quantum noise}

We use quantum relative entropy as measure of quantum information. Assume the logarithmic mean setting. Note our discussion concerning quantum optimal transport as transport of quantum information in Subsection \ref{SSEC.QOT_CG}.\par
Theorem \ref{THM.L2W_EVI_Equivalence} ensures we may use any condition in Definition \ref{DFN.L2W_EVI_Equivalence} as equivalent characterisation. Definition \ref{DFN.L2W_EVI_Ric} gives lower Ricci bounds of quantum gradients. We view them as measurement convexity of quantum information. Specifically, note $\CNV_{\lambda}$ as per Equation \ref{EQ.SSEC.L2W_EVI_Equivalence_5} shows lower Ricci bounds are given by $\lambda$-convexity of quantum information along minimising geodesics measured by quantum relative entropy. In light of our discussion in Subsection \ref{SSEC.QOT_CG}, this is a non-spatial description of $\lambda$-convexity but not one we have related to computation. If we do have noncommutative analogues of displacement interpolations \cite{ART.CorEra_McCan_Sch.2001.Displacement_Convexity_Riemannian}\cite{ART.McCan.1997.Displacement_Convexity_Local}, then precomposition with quantum channels as per Remark \ref{REM.L2W_EVI_Ric} transforms such measurement convexity in the Schr\"odinger picture into convexity under measurement of observables in the Heisenberg picture. We may view such channels as computations of a quantum computer \cite{ART.Ash_Geo_Nor.2014.Quantum_Simulation}\cite{BK.Nie_Chu.2000.Quantum_Computation_Information} to get a computational interpretation of lower Ricci bounds. Unfortunately, existence results are unknown to us. We instead show strictly positive lower Ricci bounds determine energy-information trade-offs pa\-ra\-metrised by lower bounds on quantum noise. Lower resolution implies lower energy paths. We avoid spatial interpretations of the classical case \cite{ART.Dol_Naz_Sav.2009.Generalised_OT}\cite{ART.Lot_Vil.2009.Classical_OT_Ricci_Bounds}.\par
Let $(\phi,\bpsi,\gamma,\nabla)$ be noncommutative differential structure for tracial AF-$C^{*}$-algebras $(A,\tau)$ and $(B,\omega)$ in the logarithmic mean setting.

\begin{dfn}\label{DFN.L2W_EVI_Ric}
We say that $\lambda\in\mathbb{R}$ is a lower Ricci bound on $\SII(A)$ given $(\phi,\bpsi,\gamma,\nabla)$ if any condition in Definition \ref{DFN.L2W_EVI_Equivalence} is satisfied for $\lambda$. We further write $\Ric\nabla\geq\lambda$ and say that $\lambda$ is a lower Ricci bound of $\nabla$.
\end{dfn}

\begin{rem}\label{REM.L2W_EVI_Ric}
We know lower Ricci bounds \cite{ART.Lot_Vil.2009.Classical_OT_Ricci_Bounds}\cite{ART.Stu.2006.Classical_OT_I}\cite{ART.Stu.2006.Classical_OT_II} for optimal transport on continuous geometries \cite{BK.Amb_Gig_Sav.2008.Classical_OT_GradFlow}\cite{ART.Dol_Naz_Sav.2009.Generalised_OT}\cite{BK.Vil.2009.OT} are displacement convexity of relative entropy in the sense of McCann \cite{ART.CorEra_McCan_Sch.2001.Displacement_Convexity_Riemannian}\cite{ART.McCan.1997.Displacement_Convexity_Local}. Let $\lc{}X,g\rc$ be a complete connected smooth Riemannian manifold and $d\vol$ the Riemannian density on $X$ \lc{}cf.~pp.299-306 in \cite{BK.Lan.1995.Riemannian_Manifolds}\rc{}. Get metric measure space $\lc{}X,d^{g},d\vol\rc$ with $d^{g}$ given by $g$ and exponential map $\exp:TX\longrightarrow X$ on $TX$ by the Hopf-Rinow theorem \lc{}cf.~pp.216-224 in \cite{BK.Lan.1995.Riemannian_Manifolds}\rc{}. If $\mu:[0,1]\longrightarrow\mathcal{S}^{\NI}\lc{}C_{0}(X)\rc$ is a minimising geodesic for the classical $L^{2}$-Wasserstein distance \cite{ART.Dol_Naz_Sav.2009.Generalised_OT}, then Theorem 3.2 and Corollary 5.2 in \cite{ART.CorEra_McCan_Sch.2001.Displacement_Convexity_Riemannian} imply there exists a $d\vol$-a.e.~differentiable map $u:X\longrightarrow\mathbb{R}$ and homotopy $F:[0,1]\times X\longrightarrow X$ defined by

\begin{align}\label{EQ.REM.L2W_EVI_Ric_1}
F(t)(x):=\textrm{exp}_{x}\lc{}-t\cdot \textrm{grad}_{x} u\rc{}
\end{align}

\noindent for all $x\in X$ and $t\in [0,1]$ s.t.~its dualisation $F^{*}:[0,1]\times C_{0}(X)\longrightarrow C_{0}(X)$ in the second variable satisfies

\begin{align}\label{EQ.REM.L2W_EVI_Ric_2}
\mu(t)(h)=\int_{X}h(x)d\mu(t)=\int_{X}h(x)dF(t)_{\sharp}\lc\mu(0)\rc{}=\int_{X}h\lc{}F(t)(x)\rc{}d\mu=\mu\lc{}F(t)^{*}(h)\rc{}
\end{align}

\noindent for all $h\in C_{0}(X)$ and $t\in [0,1]$. Homotopies as per Equation \ref{EQ.REM.L2W_EVI_Ric_1} extend the pointwise case in \cite{ART.McCan.1999.Polar_Factorisation_Riemannian} and are called displacement interpolations generalising terminology in the Euclidian case \cite{ART.McCan.1997.Displacement_Convexity_Local}. Functionals satisfying strong convexity, resp.~a weaker form as per $2)$ in Definition \ref{DFN.Metric_Functional_EVI_CNV}, along interpolation lines determined by Equation \ref{EQ.REM.L2W_EVI_Ric_1} are called displacement convex. Equation \ref{EQ.REM.L2W_EVI_Ric_2} is a push-forward measure representation transforming the Eulerian picture into the Lagrangian one \lc{}cf.~pp.224-225 in \cite{ART.CorEra_McCan_Sch.2001.Displacement_Convexity_Riemannian}\rc{}.\par


\pagebreak


Noncommutative analogues of Equation \ref{EQ.REM.L2W_EVI_Ric_1} are given by deforming the identity operator using quantum channels. Indeed, precomposition with any continuous function is unital and positive, ergo completely positive by commutativity \lc{}cf.~Corollary IV.3.5 in \cite{BK.Tak.1979.OpAlg_I}\rc{}. Following Remark \ref{REM.Wstar_CP_Markovian_SG}, we see analogues of homotopies as per Equation \ref{EQ.REM.L2W_EVI_Ric_1} in the AF-$C^{*}$-setting are given by $\varphi:[0,1]\times \BII(A)\longrightarrow\BII(A)$ s.t.~$\varphi(t)\in\BII(A)$ is a quantum channel for all $t\geq 0$ and $\varphi(0)=I$. If $\mu:[0,1]\longrightarrow\mathcal{S}^{\NI}(A)$ is a minimising geodesic, then we want such deformation $\varphi:[0,1]\times\BII(A)\longrightarrow\BII(A)$ of the identity operator s.t.~

\begin{align}\label{EQ.REM.L2W_EVI_Ric_3}
\mu(t)(x)=\varphi(t)^{*}(\mu)(x)=\mu\lc\varphi(t)(x)\rc{}
\end{align}

\noindent for all $x\in A$ and $t\in [0,1]$. Passing from points $x\in X$ to observables formally replaces the Lagrangian with the Heisenberg picture as we replace vectors of real numbers with bounded operators \lc{}cf.~pp.xix-xx in \cite{BK.Tak.2003.OpAlg_II}\rc{}. Since each $\varphi(t)$ in Equation \ref{EQ.REM.L2W_EVI_Ric_3} moreover describes a state change due to measurement \cite{BK.Nie_Chu.2000.Quantum_Computation_Information}\cite{ART.Dav_Lew.1970.Wstar_Quantum_Probability}\cite{ART.Kra.1971.State_Changes}\cite{BK.Ohy_Pet.1993.Rel_Ent}, i.e.~each transmits a corresponding change of information encoded in states of the given quantum system \cite{BK.Nie_Chu.2000.Quantum_Computation_Information} providing physical realisation of a quantum computer \cite{ART.Ash_Geo_Nor.2014.Quantum_Simulation}\cite{ART.Bur_Lad_Nic.2023.QC_Spin_Overview}, Equation \ref{EQ.REM.L2W_EVI_Ric_3} shows measurement convexity in the Schr\"odinger picture as per Definition \ref{DFN.L2W_EVI_Ric} is convexity under measurement of observables in the Heisenberg picture.
\end{rem}

We show conditioning in Definition \ref{DFN.L2W_Ric_FI} is determined by the underlying metric geometry as restriction to finitely supported accessibility components. Assuming strictly positive lower Ricci bounds and finitely supported fixed part, Theorem \ref{THM.QOT_Distance_AC_Rel_Ent} classifies accessibility components of normal states with finite quantum relative entropy using fixed parts. Strictly lower Ricci bounds avoid assumptions on spectral gaps required by Theorem \ref{THM.QOT_Distance_AC_L2}. We use Corollary \ref{COR.QOT_Distance_AC_Rel_Ent} to formulate energy-information trade-offs.

\begin{thm}\label{THM.QOT_Distance_AC_Rel_Ent}
Let $(\phi,\bpsi,\gamma,\nabla)$ be noncommutative differential structure for tracial AF-$C^{*}$-algebras $(A,\tau)$ and $(B,\omega)$ in the logarithmic mean setting. Assume $\Ric\nabla\geq\lambda>0$. If $\xi\in\SII(A)$ is a finitely supported fixed state, then

\begin{itemize}
\item[1)] $\mathcal{C}_{A}^{\Ent}(\xi):=\mathcal{C}_{A}(\xi)\cap\dom\Enttau=\Fix_{A}(\xi)\cap\dom\Enttau\neq\emptyset$,

\item[2)] $\mathcal{W}_{\nabla\vert\mathcal{C}_{A}^{\Ent}(\xi)\times\mathcal{C}_{A}^{\Ent}(\xi)}^{\log}$ is finite and $\mathcal{C}_{A}^{\Ent}(\xi)\subset\mathcal{C}_{A}(\xi)$ is a geodesic subspace.
\end{itemize}
\end{thm}
\begin{proof}
Let $\xi\in\SII(A)$ be a finitely supported fixed state. Let $\mathcal{C}\subset (\SII(A),\mathcal{W}_{\nabla}^{\log})$ be finitely supported with fixed part $\xi$ s.t.~$\mathcal{C}\cap\dom\Enttau\neq\emptyset$. Note $\Enttau:\mathcal{C}\cap\dom\Enttau\longrightarrow (-\infty,\infty)$ has complete sublevels in $\mathcal{W}_{\nabla}^{\log}$-topology by $3)$ in Proposition \ref{PRP.L2W_EVI_Equivalence}. We see $\Enttau$ has a unique minimum $\mu_{\min}\in\mathcal{C}\cap\dom\Enttau$ by $2)$ in Proposition \ref{PRP.Metric_Functional_EVI_CNV}. Theorem \ref{THM.L2W_Log_Mean_NCDS} yields $\mu_{\min}=\xi$ by minimality. Ergo $\mathcal{C}=\mathcal{C}_{A}(\xi)$ by uniqueness of fixed states. Using the latter in each case, we have $1)$ by decomposing $\Fix_{A}(\xi)$ as per Equation \ref{EQ.SSEC.QOT_QIT_Encoding_13}. Theorem \ref{THM.L2W_EVI_Equivalence} shows $2)$ by $2.1)$ in Definition \ref{DFN.L2W_EVI_Equivalence}.
\end{proof}


\pagebreak


\begin{cor}\label{COR.QOT_Distance_AC_Rel_Ent}
Assume $\Ric\nabla\geq\lambda>0$. If $\xi\in\SII(A)$ is a finitely supported fixed state and $\mathcal{C}\subset (\SII(A),\mathcal{W}_{\nabla}^{\log})$ is finitely supported with fixed part $\xi$, then either $\mathcal{C}=\mathcal{C}_{A}(\xi)$ or $\mathcal{C}\cap\dom\Enttau=\emptyset$.
\end{cor}
\begin{proof}
If $\mathcal{C}\cap\dom\Enttau\neq\emptyset$, then our proof of $1)$ in Theorem \ref{THM.QOT_Distance_AC_Rel_Ent} shows $\mathcal{C}=\mathcal{C}_{A}(\xi)$. If $\mathcal{C}\cap\dom\Enttau=\emptyset$, then $\xi\notin\mathcal{C}$ since $\xi\in\dom\Enttau$. As such, $\mathcal{C}\neq\mathcal{C}_{A}(\xi)$ in this case.
\end{proof}

We use strictly positive lower Ricci bounds in order to determine energy-information trade-offs parametrised by lower bounds on quantum noise. Lower resolution, i.e.~higher lower bounds on quantum noise, implies lower energy paths. We give one trade-off for each finitely supported accessibility component having non-trivial intersection with the domain of quantum relative entropy. Assume $\Ric\nabla\geq\lambda>0$. Let $\xi\in\SII(A)$ be a finitely supported fixed state. Let $\mathcal{C}\subset (\SII(A),\mathcal{W}_{\nabla}^{\log})$ be finitely supported with fixed part $\xi$ s.t.~$\mathcal{C}\cap\dom\Enttau\neq\emptyset$. Corollary \ref{COR.QOT_Distance_AC_Rel_Ent} shows $\mathcal{C}=\mathcal{C}_{A}(\xi)$.\par
Following our maximum entropy production principle in Subsection \ref{SSEC.L2W_Log_Mean_QNE}, we view quantum Laplacians as generators of quantum noise evolution. Thus applying heat flow to a state for $t>0$ introduces quantum noise. We use resolutions to define lower bounds on quantum noise. We define minimal and maximal resolution on $\mathcal{C}_{A}^{\Ent}(\xi)$ by setting

\begin{align}\label{EQ.SSEC.L2W_EVI_Ric_1}
-\infty<\rho_{A}^{\min}(\xi):=\inf_{\mu\in\mathcal{C}_{A}^{\Ent}(\xi)}\hspace{0.025cm} \Ent(\mu,\tau)<\rho_{A}^{\max}(\xi):=\sup_{\mu\in\mathcal{C}_{A}^{\Ent}(\xi)}\hspace{0.025cm} \Ent(\mu,\tau)\leq\infty.
\end{align}

\noindent Get $\rho_{A}^{\min}(\xi)=\Ent(\xi,\tau)$ by $2)$ in Theorem \ref{THM.L2W_Log_Mean_NCDS}. We say that $\rho\in \lc\rho_{A}^{\min}(\xi),\rho_{A}^{\max}(\xi)\rc$ is a resolution. For all $\rho\in \lc\rho_{A}^{\min}(\xi),\rho_{A}^{\max}(\xi)\rc$, we define the resolution surface and resolution sublevel of $\rho$ by setting

\begin{align}\label{EQ.SSEC.L2W_EVI_Ric_2}
\mathrm{R}_{A}^{\Ent}(\xi,\rho):=\lc\restr{0.925}{\Enttau}{\mathcal{C}_{A}(\xi)}\rc^{-1}(\rho),\ \mathrm{S}_{A}^{\Ent}(\xi,\rho):=\bigcup_{\rho'\leq\rho}\lc\restr{0.925}{\Enttau}{\mathcal{C}_{A}(\xi)}\rc^{-1}(\rho').
\end{align}

\noindent Each $\mathrm{R}_{A}^{\Ent}(\xi,\rho)$ is determined by all states for which $2)$ in Theorem \ref{THM.L2W_Log_Mean_NCDS} prohibits gain in quantum information above $\rho$ by reducing quantum noise. We thereby use resolutions to define lower bounds on quantum noise. Of course, each $\mathrm{S}_{A}^{\Ent}(\xi,\rho)$ is a sublevel of $\Enttau:\mathcal{C}_{A}(\xi)\longrightarrow (-\infty,\infty]$. For all $\rho\in \lc\rho_{A}^{\min}(\xi),\rho_{A}^{\max}(\xi)\rc$, $3)$ in Proposition \ref{PRP.L2W_EVI_Equivalence} and $\CNV_{\lambda}$ as per Equation \ref{EQ.SSEC.L2W_EVI_Equivalence_5} show $\mathrm{S}_{A}^{\Ent}(\xi,\rho)\subset\mathcal{C}_{A}^{\Ent}(\xi)$ is a geodesic subspace and therefore a complete geodesic length-metric space.\par
Let $\rho\in \lc\rho_{A}^{\min}(\xi),\rho_{A}^{\max}(\xi)\rc$. We obtain metric-functional system $\lc\mathrm{S}_{A}^{\Ent}(\xi,\rho),\mathcal{W}_{\nabla}^{\log},\Enttau\rc$ equipped with continuous semigroup $h:[0,\infty)\times\mathrm{S}_{A}^{\Ent}(\xi,\rho)\longrightarrow\mathrm{S}_{A}^{\Ent}(\xi,\rho)$. We define the maximal lower Ricci bound of $\nabla$ given $\rho$ by setting

\begin{align}\label{EQ.SSEC.L2W_EVI_Ric_3}
\lambda_{A}^{\max}(\xi,\rho):=\sup_{\lambda'\geq\lambda}\hspace{0.025cm} \lambda',
\end{align}

\noindent where the supremum on the right-hand side of Equation \ref{EQ.SSEC.L2W_EVI_Ric_3} is taken over all $\lambda'\geq\lambda$ s.t.~$h:[0,\infty)\times\mathrm{S}_{A}^{\Ent}(\xi,\rho)\longrightarrow\mathrm{S}_{A}^{\Ent}(\xi,\rho)$ is $\EVI_{\lambda'}$-gradient flow of $\Enttau$ in $\mathrm{S}_{A}^{\Ent}(\xi,\rho)$.\par
For all $\mu,\eta\in\mathrm{S}_{A}^{\Ent}(\xi,\rho)$, $1)$ in Proposition \ref{PRP.Metric_Functional_EVI_CNV} shows

\begin{align}\label{EQ.SSEC.L2W_EVI_Ric_4}
\mathcal{W}_{\nabla}^{\log}\lc{}h_{t}(\mu),h_{t}(\eta)\rc\leq e^{-t\lambda_{A}^{\max}(\xi,\rho)}\mathcal{W}_{\nabla}^{\log}(\mu,\eta)
\end{align}

\noindent for all $t\geq 0$. Moreover, $3)$ in Proposition \ref{PRP.Metric_Functional_EVI_CNV} shows $\Enttau$ is $\lambda_{A}^{\max}(\xi,\rho)$-convex in $\mathcal{C}_{A}^{\Ent}(\xi)$. Note $\EVI_{\lambda}$ as per Equation \ref{EQ.SSEC.L2W_EVI_Equivalence_3} shows $\lambda_{A}^{\max}(\xi,\rho)\geq\lambda>0$. Equation \ref{EQ.SSEC.L2W_EVI_Ric_4} further shows introducing quantum noise relative to $\rho$, i.e.~$t>0$, implies lower energy paths.\par
We obtain monotonically decreasing map $\lambda_{A}^{\max}\lc\xi,\blank\rc{}:\lc\rho_{A}^{\min}(\xi),\rho_{A}^{\max}(\xi)\rc\longrightarrow [\lambda,\infty)$. As such, Equation \ref{EQ.SSEC.L2W_EVI_Ric_4} shows a decrease in resolution, i.e.~from $\rho$ to $\rho'<\rho$, implies lower energy paths if $\lambda_{A}^{\max}\lc\xi,\rho'\rc{}>\lambda_{A}^{\max}(\xi,\rho)$. For all $\rho\in \lc\rho_{A}^{\min}(\xi),\rho_{A}^{\max}(\xi)\rc$, $\mu^{0},\mu^{1}\in\mathrm{S}_{A}^{\Ent}(\xi,\rho)$ and $(\mu,w)\in\Geo\lc\mu^{0},\mu^{1}\rc$ s.t.~$\mu(t)\in\dom\Enttau$ for all $t\geq 0$, we have

\begin{align}\label{EQ.SSEC.L2W_EVI_Ric_5}
\Ent\lc\mu(t),\tau\rc\leq\rho-\frac{\lambda_{A}^{\max}(\xi,\rho)}{2}t\lc{}1-t\rc\cdot \mathcal{W}_{\nabla}^{\log}\lc\mu^{0},\mu^{1}\rc^{2}
\end{align}

\noindent for all $t\in [0,1]$. Equation \ref{EQ.SSEC.L2W_EVI_Ric_5} shows we obtain lower energy paths since energy costs of introducing and reducing quantum noise along minimising geodesics are lowered if resolutions are lowered.\par
Equation \ref{EQ.SSEC.L2W_EVI_Ric_8} gives the energy-information trade-off for $\mathcal{C}_{A}^{\Ent}(\xi)$ parametrised by lower bounds on quantum noise, i.e.~by resolutions. We define strictly monotonically increasing map $\textrm{diam}_{A}^{\xi}:\lc\rho_{A}^{\min}(\xi),\rho_{A}^{\max}(\xi)\rc\longrightarrow (0,\infty)$ by setting

\begin{align}\label{EQ.SSEC.L2W_EVI_Ric_6}
\textrm{diam}_{A}^{\xi}(\rho):=\sqrt{\frac{8}{\lambda_{A}^{\max}(\xi,\rho)}\lc\rho-\rho_{\min}^{A}(\xi)\rc{}}
\end{align}

\noindent for all $\rho\in \lc\rho_{A}^{\min}(\xi),\rho_{A}^{\max}(\xi)\rc$. For all $\lc\rho_{A}^{\min}(\xi),\rho_{A}^{\max}(\xi)\rc$, Equation 3.18a in the statements on asymptotic behaviour as $t\rightarrow\infty$ as per Theorem 3.5 in \cite{ART.Mur_Sav.2020.Classical_OT_EVI} for $\lambda>0$ shows

\begin{align}\label{EQ.SSEC.L2W_EVI_Ric_7}
\mathcal{W}_{\nabla}^{\log}(\mu,\xi)\leq\sqrt{\frac{2}{\lambda_{A}^{\max}(\xi,\rho)}\lc\rho-\rho_{\min}^{A}(\xi)\rc{}}
\end{align}

\noindent for all $\mu\in\mathrm{S}_{A}^{\Ent}(\xi,\rho)$. Equation \ref{EQ.SSEC.L2W_EVI_Ric_7} is the Talagrand inequality $\TW_{\lambda}$ for $\lambda\geq 0$ as per $3)$ in Definition \ref{DFN.L2W_Ric_FI}. Using triangle inequality, Equation \ref{EQ.SSEC.L2W_EVI_Ric_6} and Equation \ref{EQ.SSEC.L2W_EVI_Ric_7} let us calculate

\begin{align}\label{EQ.SSEC.L2W_EVI_Ric_8}
\textrm{diam}\ \mathrm{S}_{A}^{\Ent}(\xi,\rho)\leq\textrm{diam}_{A}^{\xi}(\rho)
\end{align}

\noindent for all $\rho\in \lc\rho_{A}^{\min}(\xi),\rho_{A}^{\max}(\xi)\rc$. Equation \ref{EQ.SSEC.L2W_EVI_Ric_8} gives, on $\mathcal{C}_{A}^{\Ent}(\xi)$, a global description of our above discussion. Lower resolution implies lower energy paths since energy costs of introducing and reducing quantum noise along minimising geodesics are lowered if resolutions are lowered. Equation \ref{EQ.SSEC.L2W_EVI_Ric_8} formulates an energy-information trade-off since lower energy paths are obtained by introducing quantum noise.


\subsubsection*{Sufficient conditions}

Theorem \ref{THM.L2W_Ric} gives sufficient conditions for lower Ricci bounds of direct sum quantum gradients. We adapt the proof of Theorem 10.9 in \cite{ART.Car_Maa.2020.Quantum_OT_III} to the AF-$C^{*}$-setting by means of the coarse graining process. Corollary \ref{COR.L2W_Ric}, which uses Lemma \ref{LEM.L2W_Ric}, is essential for this. Lemma \ref{LEM.L2W_Ric} provides detailed proof of a necessary extension of Theorem 5 in \cite{ART.Hia_Pet.2012.Quasi_Entropy_I} to all finite-dimensional $C^{*}$-algebras. Example \ref{BSP.L2W_Ric_Wstar_Derivation_QG_Dynamic_System} and Example \ref{BSP.L2W_Ric_Wstar_Derivation_QG_Intertwining_Clifford} derive non-negative, resp.~strictly positive lower Ricci bounds.\par
We consider the following direct sum noncommutative differential structures. Let $m\in\mathbb{N}$. Let $(A,\tau)$ be a tracial AF-$C^{*}$-algebra and $(\phi,\bpsi,\gamma)$ an AF-$A$-bimodule structure on $A$. For all $n\in\lset{}1,\ldots,m\rset$, let $\partial_{n}:A_{0}\longrightarrow L^{2}(A,\tau)$ be a quantum gradient. We view each as noncommutative directional derivative. Proposition \ref{PRP.Wstar_Derivation_QG_DS_I} yields their direct sum quantum gradient $\nabla^{\oplus}=\oplus_{n=1}^{m}\partial_{n}:A_{0}\longrightarrow L^{2}\lc\oplus_{n=1}^{m}A,\oplus_{n=1}^{m}\tau\rc$. Set

\begin{align}\label{EQ.SSEC.L2W_EVI_Ric_9}
\lc\phi^{m},\bpsi^{m},\gamma^{m},\nabla^{\oplus}\rc{}:=\lc\oplus_{n=1}^{m}\phi,\oplus_{n=1}^{m}\bpsi,\oplus_{n=1}^{m}\gamma,\oplus_{n=1}^{m}\partial_{n}\rc{}    
\end{align}

\noindent for tracial AF-$C^{*}$-algebras $(A,\tau)$ and $(B,\omega):=\lc\oplus_{n=1}^{m}A,\oplus_{n=1}^{m}\tau\rc$ in the logarithmic mean setting. We use Notation \ref{NTN.DS}. For details on direct sum quantum gradients, we refer to Subsection \ref{SSEC.NCDS_NCG_QG}.

\begin{ntn}\label{NTN.L2W_Ric}
We write $\mathcal{I}_{A,A}^{\log}$ for the quasi-entropy of the canonical AF-$A$-bimodule structure on $A$ in the logarithmic mean setting. Compare to Notation \ref{NTN.QE_AF}. For all $n\in\mathbb{N}$ and tracial AF-$C^{*}$-algebra $\lc{}M_{n}(\mathbb{C}),\tr_{n}\rc$ using non-normalised canonical trace, we further write $\mathcal{I}_{n,\tr}^{\log}$ for the quasi-entropy of the canonical AF-$M_{n}(\mathbb{C})$-bimodule structure on $M_{n}(\mathbb{C})$ in the logarithmic mean setting
\end{ntn}

\begin{lem}\label{LEM.L2W_Ric}
Assume $A$ is finite-dimensional. If $\varphi:A\longrightarrow A$ is a completely positive trace-preserving map, then we have

\begin{align}\label{EQ.LEM.L2W_Ric_1}
\mathcal{I}_{A,A}^{\log}\lc\varphi\lc\sharp\mu\rc^{\flat},\varphi\lc\sharp\eta\rc^{\flat},\varphi\lc\sharp w\rc^{\flat}\rc\leq\mathcal{I}_{A,A}^{\log}(\mu,\eta,w)
\end{align}

\noindent for all $\mu,\eta\in A_{+}^{*}$ and $w\in A^{*}$.
\end{lem}
\begin{proof}
Let $n,q\in\mathbb{N}$. We consider $\lc{}M_{n}(\mathbb{C}),\tr_{n}\rc$ and $\lc{}M_{q}(\mathbb{C}),\tr_{q}\rc$ both as finite-dimensional tracial AF-$C^{*}$-algebras, as well as Hilbert spaces using GNS-inner product of their respective non-normalised canonical traces. Let $\beta:M_{n}(\mathbb{C})\longrightarrow M_{q}(\mathbb{C})$ be completely positive trace-preserving. Theorem 5 in \cite{ART.Hia_Pet.2012.Quasi_Entropy_I} shows we have 

\begin{align}\label{EQ.LEM.L2W_Ric_2}
\beta^{*}\circ\mathcal{D}_{\beta(X),\beta(Y)}\circ\beta\leq\mathcal{D}_{X,Y}
\end{align}

\noindent in $\BII\lc{}M_{n}(\mathbb{C})\rc$ for all $X,Y>0$ in $M_{n}(\mathbb{C})$. Equation \ref{EQ.LEM.L2W_Ric_2} shows

\begin{align}\label{EQ.LEM.L2W_Ric_3}
\mathcal{I}_{q,\tr}^{\log}\lc\beta(X)^{\flat},\beta\lc{}Y\rc^{\flat},\beta(U)^{\flat}\rc\leq\mathcal{I}_{n,\tr}^{\log}\lc{}X^{\flat},Y^{\flat},U^{\flat}\rc{}
\end{align}

\noindent for all $X,Y>0$ in $M_{n}(\mathbb{C})$ and $U\in M_{q}(\mathbb{C})$. We suppress sharp operators in all equations here. We show our claim by reducing Equation \ref{EQ.LEM.L2W_Ric_1} to Equation \ref{EQ.LEM.L2W_Ric_3}.\par


\pagebreak


Note $\mathcal{I}_{A,A}^{\log}$ is jointly convex and l.s.c.~in $w^{*}$-topology by $1)$ in Theorem \ref{THM.QE_AF}. We scale with strictly positive constants as in the proof of Proposition \ref{PRP.Energy_Functional_Restriction} by construction of quasi-entropies. Let $\varphi:A\longrightarrow A$ be completely positive trace-preserving. Since we know $\varphi$ is $w^{*}$-continuous by finite-dimensionality, l.s.c.~in $w^{*}$-topology implies Equation \ref{EQ.LEM.L2W_Ric_1} if it holds for all $\mu,\eta\in\SII(A)$ s.t.~$\sharp\mu,\sharp\eta>0$ in $A$. Let

\begin{align}\label{EQ.LEM.L2W_Ric_4}
(A,\tau)\overset{r_{A}}{\cong} (A',\tau'):=\lc\oplus_{l=1}^{m}M_{n_{l}}(\mathbb{C}),\oplus_{l=1}^{m}C_{l}\textrm{tr}_{n_{l}}\rc{}.
\end{align}

\noindent Equation \ref{EQ.LEM.L2W_Ric_4} uses Notation \ref{NTN.AF_Cstar_Fin_Isometry}. We know such $r_{A}$ is completely positive since it is a $^{*}$-homomorphism \lc{}cf.~Example \ref{BSP.Wstar_CP_II}\rc{}. It is furthermore trace-preserving by $2)$ in Proposition \ref{PRP.AF_Cstar_Trace_II}. We see $\varphi':=r_{A}\circ\varphi\circ r_{A}^{-1}$ is completely positive trace-preserving.\par
Proposition \ref{PRP.AF_Cstar_Trace_II} and $2)$ in Proposition \ref{PRP.NCD_Operator_Compressed_PMO_Fin} imply

\begin{align}\label{EQ.LEM.L2W_Ric_5}
\mathcal{I}_{A,A}^{\log}(x,y,u)=\mathcal{I}_{A',A'}^{\log}\lc{}r_{A}(x)^{\flat},r_{A}(y)^{\flat},r_{A}(u)^{\flat}\rc{}
\end{align}

\noindent for all $x,y>0$ in $A$ and $u\in A$. Equation \ref{EQ.LEM.L2W_Ric_5} implies Equation \ref{EQ.LEM.L2W_Ric_1} if and only if

\begin{align}\label{EQ.LEM.L2W_Ric_6}
\mathcal{I}_{A',A'}^{\log}\lc\varphi'(X)^{\flat},\varphi'\lc{}Y\rc^{\flat},\varphi'(U)^{\flat}\rc\leq\mathcal{I}_{A',A'}^{\log}\lc{}X^{\flat},Y^{\flat},U^{\flat}\rc{}
\end{align}

\noindent for all $X,Y>0$ in $A'$ and $U\in A'$. We reduce Equation \ref{EQ.LEM.L2W_Ric_6} to Equation \ref{EQ.LEM.L2W_Ric_3}.\par
We assume $(A,\tau)=(A',\tau')$ without loss of generality. Thus $r_{A}=\id_{A}$, hence $\varphi=\varphi'$. We require several identities and completely positive trace-preserving maps in order to apply Equation \ref{EQ.LEM.L2W_Ric_3}. For all $l\in\lset{}1,\ldots,m\rset$, set $X_{l}:=\pi_{l}(X)$ for all $X\in A$. The latter uses Notation \ref{NTN.DS}. Proposition \ref{PRP.AF_Cstar_Trace_II} and $2)$ in Proposition \ref{PRP.NCD_Operator_Compressed_PMO_Fin} imply

\begin{align}\label{EQ.LEM.L2W_Ric_7}
\mathcal{I}_{A,A}^{\log}\lc{}X^{\flat},Y^{\flat},U^{\flat}\rc{}=\sum_{l=1}^{m}C_{l}\mathcal{I}_{n_{l},\tr}^{\log}\lc{}X_{l}^{\flat},Y_{l}^{\flat},U_{l}^{\flat}\rc{}
\end{align}

\noindent for all $X,Y>0$ in $A$ and $U\in A$. Set $q:=\sum_{l=1}^{m}n_{l}$. We consider the diagonal $A\subset M_{q}(\mathbb{C})$. If we moreover consider $C_{l}=1$ for all $l\in\lset{}1,\ldots,m\rset$, then Equation \ref{EQ.LEM.L2W_Ric_7} yields

\begin{align}\label{EQ.LEM.L2W_Ric_8}
\mathcal{I}_{q,\tr}^{\log}\lc{}X^{\flat},Y^{\flat},U^{\flat}\rc{}=\sum_{l=1}^{m}\mathcal{I}_{n_{l},\tr}^{\log}\lc{}X_{l}^{\flat},Y_{l}^{\flat},U_{l}^{\flat}\rc{}
\end{align}

\noindent for all $X,Y>0$ in $A$, ergo $M_{q}(\mathbb{C})$, and $U\in A$. Set $M_{C}(X):=\sum_{l=1}^{m}C_{l}X_{l}$ for all $X\in A$. The direct sum construction implies $M_{C}(X)>0$ in $M_{q}(\mathbb{C})$ for all $X>0$ in $A$ as $C_{l}>0$ in each case by assumption.\par


\pagebreak


By scaling with strictly positive constants, Equation \ref{EQ.LEM.L2W_Ric_7} and Equation \ref{EQ.LEM.L2W_Ric_8} let us calculate

\begin{align*}
\mathcal{I}_{A,A}^{\log}\lc{}X^{\flat},Y^{\flat},U^{\flat}\rc{} & = \sum_{l=1}^{m}C_{l}\mathcal{I}_{n_{l},\tr}^{\log}\lc{}X_{l}^{\flat},Y_{l}^{\flat},U_{l}^{\flat}\rc \phantom{\bigg)} \\
& = \sum_{l=1}^{m}\mathcal{I}_{n_{l},\tr}^{\log}\lc{}C_{l}X_{l}^{\flat},C_{l}Y_{l}^{\flat},C_{l}U_{l}^{\flat}\rc \phantom{\bigg)} \\
& =\mathcal{I}_{q,\tr}^{\log}\lc{}M_{C}(X)^{\flat},M_{C}\lc{}Y\rc^{\flat},M_{C}(U)^{\flat}\rc \phantom{\bigg)}
\end{align*}

\noindent in each case. Precomposing with $\varphi$ in the above calculation shows

\begin{align}\label{EQ.LEM.L2W_Ric_9}
\mathcal{I}_{A,A}^{\log}\lc\varphi(X)^{\flat},\varphi\lc{}Y\rc^{\flat},\varphi(U)^{\flat}\rc{}=\mathcal{I}_{q,\tr}^{\log}\lc{}M_{C}\lc\varphi(X)\rc^{\flat},M_{C}\lc\varphi(Y)\rc^{\flat},M_{C}\lc\varphi(U)\rc^{\flat}\rc{} 
\end{align}

\noindent for all $X,Y>0$ in $A$ and $U\in A$. Altogether, we have the required identities.\par
For all $l\in\lset{}1,\ldots,m\rset$, we define $\varphi_{l}:M_{n_{l}}(\mathbb{C})\longrightarrow M_{q}(\mathbb{C})$ by setting

\begin{align}\label{EQ.LEM.L2W_Ric_10}
\varphi_{l}(X):=C_{l}^{-1}M_{C}\lc\varphi(X)\rc{}
\end{align}

\noindent for all $X\in M_{n_{l}}(\mathbb{C})$. We know the diagonal $A\subset M_{q}(\mathbb{C})$ is completely positive because it is a $^{*}$-homomorphism \lc{}cf.~Example \ref{BSP.Wstar_CP_II}\rc{}. Since $\varphi$ is as well, Equation \ref{EQ.LEM.L2W_Ric_10} readily shows each $\varphi_{l}$ is completely positive. For all $l\in\lset{}1,\ldots,m\rset$, trace-preservation of $\varphi$ implies $\tr_{q}\lc{}M_{C}\varphi(X)\rc{}=\sum_{l=1}^{m}C_{l}\tr_{n_{l}}\lc{}X_{l}\rc{}=\tau\lc\varphi(X)\rc{}=C_{l}\tr_{n_{l}}(X)$ and therefore

\begin{align}\label{EQ.LEM.L2W_Ric_11}
\textrm{tr}_{q}\lc\varphi_{l}(X)\rc{}=C_{l}^{-1}\textrm{tr}_{q}\lc{}M_{C}\varphi(X)\rc{}=C_{l}^{-1}\tau\lc\varphi(X)\rc{}=\textrm{tr}_{n_{l}}\lc{}X_{l}\rc{}
\end{align}

\noindent for all $X\in M_{n_{l}}(\mathbb{C})$. Equation \ref{EQ.LEM.L2W_Ric_11} shows each $\varphi_{l}$ is trace-preserving. The latter holds for non-normalised canonical traces on full matrix algebras. Altogether, we have completely positive trace-preserving map $\varphi_{l}:M_{n_{l}}(\mathbb{C})\longrightarrow M_{q}(\mathbb{C})$ for all $l\in\lset{}1,\ldots,m\rset$.\par
We consider our final reduction and apply Equation \ref{EQ.LEM.L2W_Ric_3}. Let $X,Y>0$ in $A$, $U\in A$ and $\lset\lambda_{l}\rset_{l=1}^{m}\subset (0,1]$ s.t.~$\sum_{l=1}^{m}\lambda_{l}=1$. Then joint convexity and scaling with strictly positive constants followed by Equation \ref{EQ.LEM.L2W_Ric_9} lets us calculate

\begin{align*}
\mathcal{I}_{A,A}^{\log}\lc\varphi(X)^{\flat},\varphi\lc{}Y\rc^{\flat},\varphi(U)^{\flat}\rc{} & = \mathcal{I}_{A,A}^{\log}\lc\hspace{0.025cm}\sum_{l=1}^{m}\lambda_{l}\varphi\lc\lambda_{l}^{-1}X_{l}\rc^{\flat},\sum_{l=1}^{m}\lambda_{l}\varphi\lc\lambda_{l}^{-1}Y_{l}\rc^{\flat},\sum_{l=1}^{m}\lambda_{l}\varphi\lc\lambda_{l}^{-1}U_{l}\rc^{\flat}\rc \phantom{\bigg)} \\
& \leq\sum_{l=1}^{m}\lambda_{l}\mathcal{I}_{A,A}^{\log}\lc\varphi\lc\lambda_{l}^{-1}X_{l}\rc^{\flat},\varphi\lc\lambda_{l}^{-1}Y_{l}\rc^{\flat},\varphi\lc\lambda_{l}^{-1}U_{l}\rc^{\flat}\rc \phantom{\bigg)} \\
& = \sum_{l=1}^{m}\mathcal{I}_{A,A}^{\log}\lc\varphi\lc{}X_{l}\rc^{\flat},\varphi\lc{}Y_{l}\rc^{\flat},\varphi\lc{}U_{l}\rc^{\flat}\rc \phantom{\bigg)} \\
& = \sum_{l=1}^{m}\mathcal{I}_{q,\tr}^{\log}\lc{}M_{C}\lc\varphi\lc{}X_{l}\rc\rc^{\flat},M_{C}\lc\varphi\lc{}Y_{l}\rc\rc^{\flat},M_{C}\lc\varphi\lc{}U_{l}\rc\rc^{\flat}\rc{}. \phantom{\bigg)}
\end{align*}

For all $l\in\lset{}1,\ldots,m\rset$, scaling with strictly positive constants implies

\begin{align}\label{EQ.LEM.L2W_Ric_12}
\mathcal{I}_{q,\tr}^{\log}\lc{}M_{C}\lc\varphi\lc{}X_{l}\rc\rc^{\flat},M_{C}\lc\varphi\lc{}Y_{l}\rc\rc^{\flat},M_{C}\lc\varphi\lc{}U_{l}\rc\rc^{\flat}\rc{}=\mathcal{I}_{q,\tr}^{\log}\lc\varphi_{l}\lc{}X_{l}\rc^{\flat},\varphi_{l}\lc{}Y_{l}\rc^{\flat},\varphi_{l}\lc{}U_{l}\rc^{\flat}\rc{}.
\end{align}

\noindent Taken together, the above calculation and Equation \ref{EQ.LEM.L2W_Ric_12} show

\begin{align}\label{EQ.LEM.L2W_Ric_13}
\mathcal{I}_{A,A}^{\log}\lc\varphi(X)^{\flat},\varphi\lc{}Y\rc^{\flat},\varphi(U)^{\flat}\rc\leq\sum_{l=1}^{m}C_{l}\mathcal{I}_{q,\tr}^{\log}\lc\varphi_{l}\lc{}X_{l}\rc^{\flat},\varphi_{l}\lc{}Y_{l}\rc^{\flat},\varphi_{l}\lc{}U_{l}\rc^{\flat}\rc{}.
\end{align}

\noindent Equation \ref{EQ.LEM.L2W_Ric_3} applies to each summand on the right-hand side of Equation \ref{EQ.LEM.L2W_Ric_13} since each $\varphi_{l}:M_{n_{l}}(\mathbb{C})\longrightarrow M_{q}(\mathbb{C})$ is completely positive trace-preserving. Using Equation \ref{EQ.LEM.L2W_Ric_3} accordingly, applying Equation \ref{EQ.LEM.L2W_Ric_13} and Equation \ref{EQ.LEM.L2W_Ric_7} in order lets us calculate

\begin{align*}
\mathcal{I}_{A,A}^{\log}\lc\varphi(X)^{\flat},\varphi\lc{}Y\rc^{\flat},\varphi(U)^{\flat}\rc{} & \leq \sum_{l=1}^{m}C_{l}\mathcal{I}_{q,\tr}^{\log}\lc\varphi_{l}\lc{}X_{l}\rc^{\flat},\varphi_{l}\lc{}Y_{l}\rc^{\flat},\varphi_{l}\lc{}U_{l}\rc^{\flat}\rc \phantom{\bigg)} \\
& \leq \sum_{l=1}^{m}C_{l}\mathcal{I}_{q,\tr}^{\log}\lc{}X_{l}^{\flat},Y_{l}^{\flat},U_{l}^{\flat}\rc \phantom{\bigg)} \\
& =\mathcal{I}_{A,A}^{\log}\lc{}X^{\flat},Y^{\flat},U^{\flat}\rc{}. \phantom{\bigg)}
\end{align*}

\noindent This yields Equation \ref{EQ.LEM.L2W_Ric_6}. The general case follows as discussed above.
\end{proof}

\begin{cor}\label{COR.L2W_Ric}
Assume $A$ is finite-dimensional. Let $\lambda\in\mathbb{R}$ and set $h_{t}^{\dagger}:=\oplus_{n=1}^{m}e^{-\lambda t}h_{t}$ in $\BII(B)$ for all $t\geq 0$. If $\lb\phi,\Delta_{n}\rb{}=\lb\bpsi,\Delta_{n}\rb{}=0$ for all $n\in\lset{}1,\ldots,m\rset$, then we have

\begin{align}\label{EQ.COR.L2W_Ric_1}
\mathcal{I}^{\log}\lc{}h_{t}(\mu),h_{t}(\eta),h_{t}^{\dagger}\lc\sharp w\rc^{\flat}\rc\leq e^{-2\lambda t}\mathcal{I}^{\log}(\mu,\eta,w)
\end{align}

\noindent for all $\mu,\eta\in A_{+}^{*}$, $w\in B^{*}$ and $t\geq 0$.
\end{cor}
\begin{proof}
We suppress sharp operators in all equations here. We show our claim by reducing Equation \ref{EQ.COR.L2W_Ric_1} to Lemma \ref{LEM.L2W_Ric}. Let $x,y\in A_{+}$, $u\in B$ and $t\geq 0$. Since $\Delta^{\oplus}=\sum_{n=1}^{m}\Delta_{n}$ by $4)$ in Proposition \ref{PRP.Wstar_Derivation_QG_DS_I}, we see $\lb\phi,\Delta_{n}\rb{}=\lb\bpsi,\Delta_{n}\rb{}=0$ for all $n\in\lset{}1,\ldots,m\rset$ implies 

\begin{align}\label{EQ.COR.L2W_Ric_2}
\lb\phi,h_{t}\rb{}=\lb\bpsi,h_{t}\rb{}=0    
\end{align}

\noindent for all $t\geq 0$. Equation \ref{EQ.COR.L2W_Ric_2} in turn shows

\begin{align}\label{EQ.COR.L2W_Ric_3}
\mathcal{I}^{\log}\lc{}h_{t}(x)^{\flat},h_{t}(y)^{\flat},h_{t}^{\dagger}(u)^{\flat}\rc{}=\mathcal{I}_{A,B}^{\log}\lc{}h_{t}\lc\phi(x)\rc^{\flat},h_{t}\lc\bpsi(y)\rc^{\flat},h_{t}^{\dagger}\lc\phi(u)\rc^{\flat}\rc{}
\end{align}

\noindent by construction of quasi-entropies.\par


\pagebreak


Moreover, Proposition \ref{PRP.Wstar_Derivation_QG_DS_II} shows

\begin{align}\label{EQ.COR.L2W_Ric_4}
\mathcal{I}_{A,B}^{\log}\lc{}h_{t}\lc\phi(x)\rc^{\flat},h_{t}\lc\bpsi(y)\rc^{\flat},h_{t}^{\dagger}\lc\phi(u)\rc^{\flat}\rc{}=\sum_{n=1}^{m}\mathcal{I}_{A,A}^{\log}\lc{}h_{t}\lc\phi(x)\rc^{\flat},h_{t}\lc\bpsi(y)\rc^{\flat},\pi_{n}\big(h_{t}^{\dagger}(u)\big)^{\flat}\rc{}.
\end{align}

\noindent We combine Equation \ref{EQ.COR.L2W_Ric_3} and Equation \ref{EQ.COR.L2W_Ric_4}. We obtain

\begin{align}\label{EQ.COR.L2W_Ric_5}
\mathcal{I}^{\log}\lc{}h_{t}(x)^{\flat},h_{t}(y)^{\flat},h_{t}^{\dagger}(u)^{\flat}\rc{}=\sum_{n=1}^{m}\mathcal{I}_{A,A}^{\log}\lc{}h_{t}\lc\phi(x)\rc^{\flat},h_{t}\lc\bpsi(y)\rc^{\flat},\pi_{n}\big(h_{t}^{\dagger}(u)\big)^{\flat}\rc{}.
\end{align}

\noindent Note $1)$ in Proposition \ref{PRP.Wstar_Derivation_QG_HSG_II} shows $h_{t}:A\longrightarrow A$ is completely positive trace-preserving. Applying Equation \ref{EQ.COR.L2W_Ric_5}, Lemma \ref{LEM.L2W_Ric} and finally Proposition \ref{PRP.Wstar_Derivation_QG_DS_II} in order lets us calculate

\begin{align*}
\mathcal{I}^{\log}\lc{}h_{t}(x)^{\flat},h_{t}(y)^{\flat},h_{t}^{\dagger}(u)^{\flat}\rc{} & = \sum_{n=1}^{m}\mathcal{I}_{A,A}^{\log}\lc{}h_{t}\lc\phi(x)\rc^{\flat},h_{t}\lc\bpsi(y)\rc^{\flat},\pi_{n}\big(h_{t}^{\dagger}(u)\big)^{\flat}\rc \phantom{\Bigg)} \\
& = e^{-2\lambda t}\cdot \sum_{n=1}^{m}\mathcal{I}_{A,A}^{f,\theta}\lc{}h_{t}\lc\phi(x)\rc^{\flat},h_{t}\lc\bpsi(y)\rc^{\flat},h_{t}\big(\pi_{n}(u)\big)^{\flat}\rc \phantom{\Bigg)} \\
& \leq e^{-2\lambda t}\cdot \sum_{n=1}^{m}\mathcal{I}^{\log}\lc{}x^{\flat},y^{\flat},\pi_{n}(u)^{\flat}\rc{} \phantom{\Bigg)} \\
& = e^{-2\lambda t}\cdot \mathcal{I}^{\log}\lc{}x^{\flat},y^{\flat},u^{\flat}\rc{}. \phantom{\Bigg)}
\end{align*}

\noindent The above calculation shows Equation \ref{EQ.COR.L2W_Ric_1}.
\end{proof}

\begin{dfn}\label{DFN.L2W_Ric}
We call $\lc\phi^{m},\bpsi^{m},\gamma^{m},\nabla^{\oplus}\rc$ as per Equation \ref{EQ.SSEC.L2W_EVI_Ric_9} their direct sum noncommutative differential structure. Let $\lambda\in\mathbb{R}$. If

\begin{itemize}
\item[1)] $\lb\phi,\Delta_{n}\rb{}=\lb\bpsi,\Delta_{n}\rb{}=0$,

\item[2)] $\partial_{n}\Delta^{\oplus}=\big(\Delta^{\oplus}+\lambda\cdot I\big)\partial_{n}$,
\end{itemize}

\noindent on $A_{0}$ for all $n\in\lset{}1,\ldots,m\rset$, then we say that $\nabla^{\oplus}$ is $\lambda$-intertwining.
\end{dfn}

\begin{thm}\label{THM.L2W_Ric}
Let $m\in\mathbb{N}$. Let $(A,\tau)$ be a tracial AF-$C^{*}$-algebra and $(\phi,\bpsi,\gamma)$ an AF-$A$-bimodule structure on $A$. For all $n\in\lset{}1,\ldots,m\rset$, let $\partial_{n}:A_{0}\longrightarrow L^{2}(A,\tau)$ be a quantum gradient. We consider their direct sum noncommutative differential structure. If $\nabla^{\oplus}$ is $\lambda$-intertwining, then $\Ric\nabla^{\oplus}\geq\lambda$.
\end{thm}
\begin{proof}
We adapt the proof of Theorem 10.9 in \cite{ART.Car_Maa.2020.Quantum_OT_III} to the AF-$C^{*}$-setting by means of the coarse graining process. We reduce to the finite-dimensional setting. This is necessary to apply Corollary \ref{COR.L2W_Ric}. Theorem \ref{THM.L2W_EVI_Equivalence} ensures $\HI\rc$ in Definition \ref{DFN.L2W_EVI_Equivalence} is a condition for lower Ricci bounds. For a.e.~$j\in\mathbb{N}$, note $\HI\rc$ and Definition \ref{DFN.L2W_Ric} restrict to the induced noncommutative differential structure $\lc\phi_{j}^{m},\bpsi_{j}^{m},\gamma_{j}^{m},\oplus_{n=1}^{m}\partial_{n,j}\rc$ without changing $\lambda$.\par
Assume $A$ is finite-dimensional. Then $B$ is finite-dimensional. It suffices to show $\HI_{\lambda}$ for all fixed states $\xi\in\SII(A)$, $\mu\in\vartheta(\xi)$ and $\eta\in I\lc\Delta_{\xi}^{\oplus}\rc^{\flat}$. Using Corollary \ref{COR.QOT_Distance_AC_L2}, we readily see Theorem 3.3 in \cite{ART.Dan_Sav.2008.Classical_OT_GradFlow_DisConvex} and Lemma \ref{LEM.L2W_EVI_Equivalence} show the latter if $h:[0,\infty)\times\vartheta(\xi)\longrightarrow\vartheta(\xi)$ is $\EVI_{\lambda}$-gradient flow of $\Enttau$ in $\vartheta(\xi)$ for all fixed states $\xi\in\SII(A)$. We further reduce as follows. For all fixed states $\xi\in\SII(A)$, we claim

\begin{align}\label{EQ.THM.L2W_Ric_1}
\frac{1}{2}\restr{0.925}{\frac{d^{+}}{ds}}{s=0}\mathcal{W}_{\nabla^{\oplus}}^{\log}\lc\mu,h_{s}(\eta)\rc^{2}+\frac{\lambda}{2}\mathcal{W}_{\nabla^{\oplus}}^{\log}(\mu,\eta)^{2}\leq\Ent(\mu,\tau)-\Ent(\eta,\tau)
\end{align}

\noindent for all $\mu,\eta\in\vartheta(\xi)$. If Equation \ref{EQ.THM.L2W_Ric_1} holds, then symmetry of distances, the semigroup property and Equation \ref{EQ.THM.L2W_Ric_1} itself let us calculate

\begin{align*}
\frac{1}{2}\frac{d^{+}}{dt}\mathcal{W}_{\nabla^{\oplus}}^{\log}\lc{}h_{t}(\mu),\eta\rc^{2} & = \frac{1}{2}\restr{0.925}{\frac{d^{+}}{ds}}{s=0}\mathcal{W}_{\nabla^{\oplus}}^{\log}\lc\eta,h_{s}\lc{}h_{t}(\mu)\rc\rc^{2} \phantom{\bigg)} \\
& \leq\Ent(\eta,\tau)-\Ent\lc{}h_{t}(\mu),\tau\rc{}-\frac{\lambda}{2}\mathcal{W}_{\nabla^{\oplus}}^{\log}\lc{}h_{t}(\mu),\eta\rc{} \phantom{\bigg)}
\end{align*}

\noindent for all $t\geq 0$ in each case. The above calculation in turn shows $h:[0,\infty)\times\vartheta(\xi)\longrightarrow\vartheta(\xi)$ is $\EVI_{\lambda}$-gradient flow of $\Enttau$ in $\vartheta(\xi)$ for all fixed states $\xi\in\SII(A)$. In summary, it suffices to show Equation \ref{EQ.THM.L2W_Ric_1}.\par 
We require several identities in order to show Equation \ref{EQ.THM.L2W_Ric_1}. Set $h_{t}^{\dagger}:=\oplus_{n=1}^{m}e^{-\lambda t}h_{t}$ in $\BII(B)$ for all $t\geq 0$. The latter extends to $B^{*}=\oplus_{n=1}^{m}A^{*}$ by dualisation in each summand as per construction of noncommutative heat semigroups. Note $1)$ in Definition \ref{DFN.L2W_Ric} ensures Corollary \ref{COR.L2W_Ric} applies. Using the latter, we have

\begin{align}\label{EQ.THM.L2W_Ric_2}
\mathcal{I}^{\log}\lc{}h_{t}(\mu),h_{t}(\eta),h_{t}^{\dagger}(w)\rc\leq e^{-2\lambda t}\mathcal{I}^{\log}(\mu,\eta,w)
\end{align}

\noindent for all $\mu,\eta\in\SII(A)$, $w\in B^{*}$ and $t\geq 0$. We dualise $2)$ in Definition \ref{DFN.L2W_Ric} by taking adjoints. Using the latter, $3)$ in Proposition \ref{PRP.Wstar_Derivation_QG_DS_I} implies

\begin{align}\label{EQ.THM.L2W_Ric_3}
h_{t}\nabla^{\oplus,*}=\nabla^{\oplus,*}h_{t}^{\dagger}
\end{align}

\noindent for all $t\geq 0$. Altogether, we have the required identities.\par
We show Equation \ref{EQ.THM.L2W_Ric_1}. Let $\xi\in\SII(A)$ be a fixed state. Let $\mu^{0},\mu^{1}\in\vartheta(\xi)$. Heat flow varies minimising geodesics as follows. Let $\mu:[0,1]\longrightarrow\vartheta(\xi)$ be a minimising geodesic from $\mu^{0}$ to $\mu^{1}$. Proposition \ref{PRP.RM_III} shows the canonical vector field along $\mu$ is given by $w_{t}:=\Theta\lc\mu(t),\dot{\mu}(t)\rc$ for all $t\geq 0$. We have $(\mu,w)\in\Admnullone\lc\mu^{0},\mu^{1}\rc$. Minimising geodesics have $t$-a.e.~constant speed by $1)$ in Proposition \ref{PRP.QOT_Distance_Geodesics}. The latter lets us calculate

\begin{align}\label{EQ.THM.L2W_Ric_4}
E^{\log}(\mu,w)=\mathcal{I}^{\log}\lc\mu(t),\mu(t),w(t)\rc{}
\end{align}

\noindent for all $t\in [0,1]$.\par
For all $s\in [0,1]$, set

\begin{align}\label{EQ.THM.L2W_Ric_5}
\mu_{s}(t):=h_{ts}\lc\mu(t)\rc{},\ w_{s}(t):=h_{ts}^{\dagger}\lc{}w(t)\rc{}-s\lc\nabla^{\oplus}\sharp\mu_{s}(t)\rc^{\flat}
\end{align}

\noindent for all $t\in [0,1]$. For all $s\in [0,1]$, Equation \ref{EQ.THM.L2W_Ric_3} and Equation \ref{EQ.THM.L2W_Ric_5} let us calculate

\begin{align*}
\dfrac{d}{dt}\mu_{s}(t) & = h_{ts}\lc\nabla^{\oplus,*}\sharp w(t)\rc^{\flat}-s\lc\Delta^{\oplus}\sharp\mu_{s}(t)\rc^{\flat} \phantom{\bigg)} \\ 
& = \lc\nabla^{\oplus,*}\lc\sharp h_{ts}^{\dagger}\lc{}w(t)\rc{}-s\nabla^{\oplus}\sharp\mu_{s}(t)\rc\rc^{\flat}=\lc\nabla^{\oplus,*}\sharp w_{s}(t)\rc^{\flat} \phantom{\bigg)}
\end{align*}

\noindent for all $t\in (0,1)$. The above calculation shows $\lc\mu_{s},w_{s}\rc\in\Admnullone\lc\mu^{0},h_{s}\lc\mu^{1}\rc\rc$ for all $s\in [0,1]$.\par
We estimate $E^{\log}\lc\mu^{s},w^{s}\rc$ in each case. Let $s\in (0,1]$. Set

\begin{align}\label{EQ.THM.L2W_Ric_7}
F_{s}(t):=-2s\lgl\mathcal{D}_{\sharp\mu_{s}(t),\xi}\sharp h_{ts}^{\dagger}\lc{}w(t)\rc{},\nabla^{\oplus}\sharp\mu_{s}(t)\rgl_{\omega}+s^{2}\dblv{}\mathcal{D}_{\sharp\mu_{s}(t),\xi}^{\frac{1}{2}}\nabla^{\oplus}\sharp\mu_{s}(t)\dblv_{\omega}^{2}
\end{align}

\noindent for all $t\in (0,1]$. Equation \ref{EQ.THM.L2W_Ric_2}, Equation \ref{EQ.THM.L2W_Ric_4} and Equation \ref{EQ.THM.L2W_Ric_5} let us calculate

\begin{align}\label{EQ.THM.L2W_Ric_8}
E^{\log}\lc\mu_{s},w_{s}\rc\leq\int_{0}^{1}e^{-2\lambda ts}dt\cdot E^{\log}(\mu,w)+\int_{0}^{1}F_{s}(t)dt.
\end{align}

\noindent We therefore define the integrand $F_{s}$ precisely as per Equation \ref{EQ.THM.L2W_Ric_7} in order to have Equation \ref{EQ.THM.L2W_Ric_8}. We directly verify

\begin{align}\label{EQ.THM.L2W_Ric_9}
\int_{0}^{1}e^{-2\lambda ts}dt=\frac{1-e^{-2\lambda s}}{2\lambda s}.    
\end{align}

\noindent Taken together, Equation \ref{EQ.THM.L2W_Ric_8} and Equation \ref{EQ.THM.L2W_Ric_9} show

\begin{align}\label{EQ.THM.L2W_Ric_10}
E^{\log}\lc\mu_{s},w_{s}\rc\leq\frac{1-e^{-2\lambda s}}{2\lambda s}\cdot E^{\log}(\mu,w)+\int_{0}^{1}F_{s}(t)dt.
\end{align}

Equation \ref{EQ.THM.L2W_Ric_10} clearly shows we must estimate the integrand $F_{s}$. Using $\sharp h_{ts}^{\dagger}\lc{}w(t)\rc{}=\sharp w_{s}(t)+s\nabla^{\oplus}\sharp\mu_{s}(t)$ in each case, $2)$ in Lemma \ref{LEM.L2W_Log_Mean_NCDS} lets us calculate

\begin{align*}
\lgl\mathcal{D}_{\sharp\mu_{s}(t),\xi}\sharp h_{ts}^{\dagger}\lc{}w(t)\rc{},\nabla^{\oplus}\sharp\mu_{s}(t)\rgl_{\omega} & = \lgl\mathcal{D}_{\sharp\mu_{s}(t),\xi}\sharp w_{s}(t),\nabla^{\oplus}\sharp\mu_{s}(t)\rgl_{\omega}+s\dblv{}\mathcal{D}_{\sharp\mu_{s}(t),\xi}^{\frac{1}{2}}\nabla^{\oplus}\sharp\mu_{s}(t)\dblv_{\omega}^{2} \phantom{\bigg)} \\
& =\frac{d}{dt}\Enttau\lc\mu_{s}(t)\rc{}+s\dblv{}\mathcal{D}_{\sharp\mu_{s}(t),\xi}^{\frac{1}{2}}\nabla^{\oplus}\sharp\mu_{s}(t)\dblv_{\omega}^{2} \phantom{\bigg)}
\end{align*}

\noindent for all $t\in (0,1)$.\par


\pagebreak


The above calculation shows

\begin{align}\label{EQ.THM.L2W_Ric_11}
-2s\frac{d}{dt}\Enttau\lc\mu_{s}(t)\rc{}=-2\lgl\mathcal{D}_{\sharp\mu_{s}(t),\xi}\sharp h_{ts}^{\dagger}\lc{}w(t)\rc{},\nabla^{\oplus}\sharp\mu_{s}(t)\rgl_{\omega}+2s^{2}\dblv{}\mathcal{D}_{\sharp\mu_{s}(t),\xi}^{\frac{1}{2}}\nabla^{\oplus}\sharp\mu_{s}(t)\dblv_{\omega}^{2}\phantom{\bigg)}
\end{align}

\noindent in each case by rearranging terms accordingly. Finally, we readily see Equation \ref{EQ.THM.L2W_Ric_7} and Equation \ref{EQ.THM.L2W_Ric_11} let us calculate

\begin{align}\label{EQ.THM.L2W_Ric_12}
F_{s}(t)=-2s\frac{d}{dt}\Enttau\lc\mu_{s}(t)\rc{}-s^{2}\dblv{}\mathcal{D}_{\sharp\mu_{s}(t),\xi}^{\frac{1}{2}}\nabla^{\oplus}\sharp\mu_{s}(t)\dblv_{\omega}^{2}\leq -2s\frac{d}{dt}\Enttau\lc\mu_{s}(t)\rc{}
\end{align}

\noindent for all $t\in (0,1)$. We combine Equation \ref{EQ.THM.L2W_Ric_10} and Equation \ref{EQ.THM.L2W_Ric_12}. We obtain

\begin{align}\label{EQ.THM.L2W_Ric_13}
E^{\log}\lc\mu_{s},w_{s}\rc\leq\frac{1-e^{-2\lambda s}}{2\lambda s}E^{\log}(\mu,w)+s\cdot \bigg(\hspace{-0.028975cm} \Ent\lc\mu^{0},\tau\rc{}-\Ent\lc\mu^{1},\tau\rc\bigg).
\end{align}

Equation \ref{EQ.THM.L2W_Ric_13} yields our required estimate of $E^{\log}\lc\mu^{s},w^{s}\rc$ for all $s\in (0,1]$. We engage in our final estimate. Equation \ref{EQ.THM.L2W_Ric_13} implies

\begin{align}\label{EQ.THM.L2W_Ric_14}
\frac{1}{2}\mathcal{W}_{\nabla^{\oplus}}^{\log}\lc\mu^{0},h_{s}\lc\mu^{1}\rc\rc^{2}\leq\frac{1-e^{-2\lambda s}}{4\lambda s}\cdot E^{\log}(\mu,w)+s\cdot \bigg(\hspace{-0.028975cm} \Ent\lc\mu^{0},\tau\rc{}-\Ent\lc\mu^{1},\tau\rc\bigg)
\end{align}

\noindent for all $s\in (0,1]$. We use the energy $E^{\log}(\mu,w)=\mathcal{W}_{\nabla^{\oplus}}^{\log}\lc\mu^{0},\mu^{1}\rc^{2}$ of the minimising geodesic $\mu:[0,1]\longrightarrow\vartheta(\xi)$ from $\mu^{0}$ to $\mu^{1}$. Corollary \ref{COR.RM} ensures we have sufficient minimising geodesics in $\vartheta(\xi)$. Equation \ref{EQ.THM.L2W_Ric_14} therefore equals

\begin{align}\label{EQ.THM.L2W_Ric_15}
\frac{1}{2}\mathcal{W}_{\nabla^{\oplus}}^{\log}\lc\mu^{0},h_{s}\lc\mu^{1}\rc\rc^{2}\leq\frac{1-e^{-2\lambda s}}{4\lambda s}\cdot \mathcal{W}_{\nabla^{\oplus}}^{\log}\lc\mu^{0},\mu^{1}\rc^{2}+s\cdot \bigg(\hspace{-0.028975cm} \Ent\lc\mu^{0},\tau\rc{}-\Ent\lc\mu^{1},\tau\rc\bigg)
\end{align}

\noindent for all $s\in (0,1]$. We see multiplying with $s^{-1}$ on and subtracting $\frac{1}{2s}\mathcal{W}_{\nabla^{\oplus}}^{\log}\lc\mu^{0},\mu^{1}\rc^{2}$ from both sides of Equation \ref{EQ.THM.L2W_Ric_15} yields

\begin{align*}
& \ \frac{1}{2}s^{-1}\lc\mathcal{W}_{\nabla^{\oplus}}^{\log}\lc\mu^{0},h_{s}\lc\mu^{1}\rc\rc^{2}-\mathcal{W}_{\nabla^{\oplus}}^{\log}\lc\mu^{0},\mu^{1}\rc^{2}\rc \phantom{\Bigg)} \\
\leq& \ \frac{1-e^{-2\lambda s}-2\lambda s}{4\lambda s^{2}}\cdot \mathcal{W}_{\nabla^{\oplus}}^{\log}\lc\mu^{0},\mu^{1}\rc^{2}+\Ent\lc\mu^{0},\tau\rc{}-\Ent\lc{}h_{s}\lc\mu^{1}\rc{},\tau\rc{} \phantom{\Bigg)}
\end{align*}

\noindent for all $s\in (0,1]$. We directly verify

\begin{align}\label{EQ.THM.L2W_Ric_16}
\lim_{s\downarrow 0}\hspace{0.025cm} \frac{1-e^{-2\lambda s}-2\lambda s}{4\lambda s^{2}}=-\frac{\lambda}{2}.
\end{align}
 
\noindent Note Equation \ref{EQ.THM.L2W_Ric_16} shows letting $s\downarrow 0$ in the final estimate yields Equation \ref{EQ.THM.L2W_Ric_1}. The general case follows as discussed above.
\end{proof}

Example \ref{BSP.L2W_Ric_Wstar_Derivation_QG_Dynamic_System} and Example \ref{BSP.L2W_Ric_Wstar_Derivation_QG_Intertwining_Clifford} derive non-negative, resp.~strictly positive lower Ricci bounds. Whereas Example \ref{BSP.L2W_Ric_Wstar_Derivation_QG_Dynamic_System} covers Example \ref{BSP.QOT_Type_I}, Example \ref{BSP.QOT_Type_II_1} and Example \ref{BSP.QOT_Type_II_Infty} in Subsection \ref{SSEC.QOT_DT_BSP}, Example \ref{BSP.L2W_Ric_Wstar_Derivation_QG_Intertwining_Clifford} covers Example \ref{BSP.QOT_Type_II_Twisted} therein.

\begin{bsp}\label{BSP.L2W_Ric_Wstar_Derivation_QG_Dynamic_System}
Assume the following setting. Let $(A,\tau)$ be a tracial AF-$C^{*}$-algebra and $\lc{}A,\mathbb{R},\alpha\rc$ a $\tau$-preserving local $C^{*}$-dynamical system. We equip $A$ with its canonical AF-$A$-bimodule structure. We use $m=1$. Corollary \ref{COR.Wstar_Derivation_QG_Dynamic_System} yields non-twisted dynamic quantum gradient $\nabla^{\DII_{\alpha},\id_{A}}$ and shows

\begin{align}\label{EQ.BSP.L2W_Ric_Wstar_Derivation_QG_Dynamic_System_Flat_1}
\Delta^{\mathcal{D}_{\alpha}}x=-\lc\nabla^{\mathcal{D}_{\alpha}}\rc^{2}(x)
\end{align}

\noindent for all $x\in A_{0}$. Equation \ref{EQ.BSP.L2W_Ric_Wstar_Derivation_QG_Dynamic_System_Flat_1} shows $\nabla^{\DII_{\alpha},\id_{A}}$ is $\lambda$-intertwining for $\lambda=0$. Theorem \ref{THM.L2W_Ric} implies $\Ric\nabla^{\DII_{\alpha},\id_{A}}\geq 0$ as claimed.
\end{bsp}

\begin{bsp}\label{BSP.L2W_Ric_Wstar_Derivation_QG_Intertwining_Clifford}
Assume the following setting. Let $(A,\tau)$ be a tracial AF-$C^{*}$-algebra and $\phi:A\longrightarrow A$ a self-adjoint involutive local $^{*}$-homomorphism. Let $m\in\mathbb{N}$ and further $\lset{}d_{n}\rset_{n=1}^{m}\subset L^{\infty}(A,\tau)_{h}$ be a $\phi$-intertwining set of Clifford generators for $C>0$ as per $1)$ in Definition \ref{DFN.Wstar_Derivation_QG_Intertwining_Clifford}. For all $n\in\lset{}1,\ldots,m\rset$, Corollary \ref{COR.Wstar_Derivation_QG_Intertwining_I} yields twisted dynamic quantum gradient $\partial_{n}=\nabla^{-iL_{d_{n}},\phi}$ and its Laplacian $\Delta_{n}=\mathrlap{\phantom{\partial}^{*}}\partial_{n}\partial_{n}$ as per $2)$ in Definition \ref{DFN.Wstar_Derivation_QG_Intertwining_Clifford}.\par
Note Equation \ref{EQ.LEM.Wstar_Derivation_QG_Intertwining_Clifford_Identities_5} lets us calculate $1)$ in Definition \ref{DFN.L2W_Ric}. Since $\Delta^{\oplus}=\sum_{n=1}^{m}\Delta_{n}$ by $4)$ in Proposition \ref{PRP.Wstar_Derivation_QG_DS_I}, Lemma \ref{LEM.Wstar_Derivation_QG_Intertwining_Clifford} implies

\begin{align}\label{EQ.BSP.L2W_Ric_Wstar_Derivation_QG_Intertwining_Clifford_1}
\partial_{n}\Delta^{\oplus}=\big(\Delta^{\oplus}+4C\cdot I\big)\partial_{n}
\end{align}

\noindent for all $n\in\lset{}1,\ldots,m\rset$. Equation \ref{EQ.BSP.L2W_Ric_Wstar_Derivation_QG_Intertwining_Clifford_1} shows $2)$ in Definition \ref{DFN.L2W_Ric}. Altogether, we see $\nabla$ is $\lambda$-intertwining for $\lambda=4C$. Theorem \ref{THM.L2W_Ric} implies $\Ric\nabla\geq 4C>0$ as claimed.
\end{bsp}


\subsubsection*{Functional inequalities}

Assuming lower Ricci bounds, Theorem \ref{THM.L2W_Ric_FI} derives functional inequalities $\HWI_{\lambda}$, $\MLSI_{\lambda}$ and $\TW_{\lambda}$ as per Definition \ref{DFN.L2W_Ric_FI}. Non-ergodicity requires relative entropy of finitely supported fixed states in their formulation. We introduce quantum Fisher information in the AF-$C^{*}$-setting. Its r\^ole mirrors the classical case \cite{ART.Lot_Vil.2009.Classical_OT_Ricci_Bounds}\cite{ART.Ott_Vil.2000.Classical_OT_LogSobolev_Talagrand}. We adapt the proofs of Theorem 11.3, Theorem 11.4 and Theorem 11.5 in \cite{ART.Car_Maa.2020.Quantum_OT_III} to the AF-$C^{*}$-setting by means of the coarse graining process. Lemma \ref{LEM.L2W_Ric_FI} provides sufficient control of metric derivatives using quantum Fisher information.\par
Let $(\phi,\bpsi,\gamma,\nabla)$ be noncommutative differential structure for tracial AF-$C^{*}$-algebras $(A,\tau)$ and $(B,\omega)$ in the logarithmic mean setting. Definition \ref{DFN.QF_AF} gives quantum Fisher information. Indeed, note $3)$ in Theorem \ref{THM.L2W_Log_Mean_QNE} and $3)$ in Proposition \ref{PRP.QF_AF} imply it is a noncommutative analogue for parametrisations $\lset{}h_{t}(\mu)\rset_{t\geq 0}$ given $\mu\in\mathcal{S}^{\NI}(A)$.

\begin{dfn}\label{DFN.QF_AF}
We define quantum Fisher information $\Ilog:A_{+}^{*}\longrightarrow [0,\infty]$ by setting

\begin{align}\label{EQ.DFN.QF_AF_1}
\Ilog(\mu):=\sup_{j\in\mathbb{N}}\hspace{0.025cm} \mathcal{I}_{j}^{\log}\lc\mu_{j},\mu_{j},\lc\nabla\sharp\mu_{j}\rc^{\flat}\rc{}
\end{align}

\noindent for all $\mu\in A_{+}^{*}$.
\end{dfn}

\begin{prp}\label{PRP.QF_AF}\hspace{1cm}
\begin{itemize}
\item[1)] $\Ilog$ is convex and l.s.c.~in $w^{*}$-topology.

\item[2)] For all $\mu\in A_{+}^{*}$, we have

\begin{itemize}
\item[2.1)] $\Ilog\lc\bar{\mu}_{j}\rc{}=\mathcal{I}_{j}^{\log}\lc\bar{\mu}_{j},\bar{\mu}_{j},\lc\nabla\sharp\bar{\mu}_{j}\rc^{\flat}\rc$ for all $j\in\mathbb{N}$, \phantom{\bigg)}

\item[2.2)] $\Ilog(\mu)=\lim_{j\in\mathbb{N}}\Ilog\lc\bar{\mu}_{j}\rc$. \phantom{\bigg)}
\end{itemize}

\begin{reapply}
\end{reapply}

\item[3)] For all finitely supported fixed states $\xi\in\SII(A)$, we have

\begin{align}\label{EQ.PRP.QF_AF_1}
\Ilog(\mu)=-\restr{0.925}{\frac{d}{dt}}{t=0}\Enttau\lc{}h_{t}(\mu)\rc{}
\end{align}

\begin{reapply}
\end{reapply}

\noindent for all $\mu\in\Fix_{A}^{\NI}(\xi)\cap\mathcal{S}_{-1}^{\NI,\infty}(A_{\xi})\cap\lc\dom\Delta\rc^{\flat}$.
\end{itemize} 
\end{prp}
\begin{proof}
We have $1)$ and $2.1)$ by $1)$, resp.~$2)$ in Theorem \ref{THM.QE_AF}. Using $2.1)$, Equation \ref{EQ.DFN.QF_AF_1} shows $3)$ in Theorem \ref{THM.QE_AF} implies $2.2)$. We show $3)$. Let $\xi\in\SII(A)$ be a finitely supported fixed state. Using $2.1)$ and $4.1)$ in Proposition \ref{PRP.Wstar_Derivation_QG_I}, note $2)$ lets us calculate

\begin{align}\label{EQ.PRP.QF_AF_2}
\Ilog(\mu)=\lim_{j\in\mathbb{N}}\hspace{0.025cm} \Ilog\lc\bar{\mu}_{j}\rc{}=\mathcal{I}^{\log}\lc\mu,\mu,\lc\nabla\sharp\mu\rc^{\flat}\rc{}
\end{align}

\noindent for all $\mu\in\Fix_{A}^{\NI}(\xi)\cap\mathcal{S}_{-1}^{\NI,\infty}(A_{\xi})\cap\lc\dom\Delta\rc^{\flat}$. The second identity in Equation \ref{EQ.PRP.QF_AF_2} uses $\sharp\mu\in\dom\Delta\subset\dom\nabla$ and therefore $\nabla\sharp\mu=\|.\|_{\omega}$-$\lim_{j\in\mathbb{N}}\nabla\sharp\bar{\mu}_{j}$ in each case. Equation \ref{EQ.PRP.QF_AF_2} shows $3)$ in Theorem \ref{THM.L2W_Log_Mean_QNE} implies $3)$ by differentiation at $t=0$.
\end{proof}

\begin{dfn}\label{DFN.L2W_Ric_FI}
Let $\lambda\in\mathbb{R}$

\begin{itemize}
\item[1)] We say that $\Enttau$ satisfies $\HWI_{\lambda}$ if for all finitely supported fixed states $\xi\in\SII(A)$ and $\mathcal{C}\subset \big(\hspace{-0.03875cm} \SII(A),\mathcal{W}_{\nabla}^{\log}\big)$ with fixed part $\xi$ s.t.~$\mathcal{C}\cap\dom\Enttau\neq\emptyset$, we have

\begin{align}\label{EQ.DFN.L2W_Ric_FI_1}
\Ent(\mu,\tau)\leq\mathcal{W}_{\nabla}^{\log}(\mu,\xi)\sqrt{\Ilog(\mu)}-\frac{\lambda}{2}\mathcal{W}_{\nabla}^{\log}(\mu,\xi)^{2}+\Ent(\xi,\tau) \tag{$\HWI_{\lambda}$}
\end{align}

\begin{reapply}
\end{reapply}

\noindent for all $\mu\in\mathcal{C}$.

\item[2)] Assume $\lambda>0$. We say that $\Enttau$ satisfies $\MLSI_{\lambda}$ if for all finitely supported fixed states $\xi\in\SII(A)$ and $\mathcal{C}\subset \big(\hspace{-0.03875cm} \SII(A),\mathcal{W}_{\nabla}^{\log}\big)$ with fixed part $\xi$ s.t.~$\mathcal{C}\cap\dom\Enttau\neq\emptyset$, we have

\begin{align}\label{EQ.DFN.L2W_Ric_FI_2}
\Ent(\mu,\tau)\leq\frac{1}{2\lambda}\Ilog(\mu)+\Ent(\xi,\tau) \tag{$\MLSI_{\lambda}$}
\end{align}

\begin{reapply}
\end{reapply}

\noindent for all $\mu\in\mathcal{C}$.

\item[3)] Assume $\lambda>0$. We say that $\Enttau$ satisfies $\TW_{\lambda}$ if for all finitely supported fixed states $\xi\in\SII(A)$ and $\mathcal{C}\subset \big(\hspace{-0.03875cm} \SII(A),\mathcal{W}_{\nabla}^{\log}\big)$ with fixed part $\xi$ s.t.~$\mathcal{C}\cap\dom\Enttau\neq\emptyset$, we have

\begin{align}\label{EQ.DFN.L2W_Ric_FI_3}
\mathcal{W}_{\nabla}^{\log}(\mu,\xi)\leq\sqrt{\frac{2}{\lambda}\lc\Ent(\mu,\tau)-\Ent(\xi,\tau)\rc{}} \tag{$\TW_{\lambda}$}
\end{align}

\begin{reapply}
\end{reapply}

\noindent for all $\mu\in\mathcal{C}$.
\end{itemize}
\end{dfn}

\begin{lem}\label{LEM.L2W_Ric_FI}
For all $\mu,\eta\in\SII(A)$, we have

\begin{align}\label{EQ.LEM.L2W_Ric_FI_1}
\limsup_{j\in\mathbb{N}}\hspace{0.025cm} \frac{d^{+}}{dt}\mathcal{W}_{\nabla}^{\log}\lc{}h_{t}\lc\bar{\mu}_{j}\rc{},\bar{\eta}_{j}\rc\leq\sqrt{\Ilog\lc{}h_{t}(\mu)\rc{}}
\end{align}

\noindent for all $t\geq 0$.
\end{lem}
\begin{proof}
We adapt the proof of Proposition 11.2 in \cite{ART.Car_Maa.2020.Quantum_OT_III} to the AF-$C^{*}$-setting by means of the coarse graining process. We reduce to the finite-dimensional setting. Note $2.2)$ in Proposition \ref{PRP.Wstar_Derivation_QG_HSG_II} reduces to Equation \ref{EQ.REM.Wstar_Derivation_QG_HSG_L2_1} in the square integrable case. For all $\mu\in\SII(A)$, $2.2)$ in Proposition \ref{PRP.QF_AF} therefore shows

\begin{align}\label{EQ.LEM.L2W_Ric_FI_2}
\Ilog\lc{}h_{t}(\mu)\rc{}=\lim_{j\in\mathbb{N}}\hspace{0.025cm} \Ilog\lc{}h_{t}\lc\bar{\mu}_{j}\rc\rc{}=\limsup_{j\in\mathbb{N}}\hspace{0.025cm} \Ilog\lc{}h_{t}\lc\bar{\mu}_{j}\rc\rc{}
\end{align}

\noindent for all $t\geq 0$. Equation \ref{EQ.LEM.L2W_Ric_FI_2} implies Equation \ref{EQ.LEM.L2W_Ric_FI_1} if for all $\mu,\eta\in\SII(A)$, we have

\begin{align}\label{EQ.LEM.L2W_Ric_FI_3}
\frac{d^{+}}{dt}\mathcal{W}_{\nabla}^{\log}\lc{}h_{t}\lc\bar{\mu}_{j}\rc{},\bar{\eta}_{j}\rc\leq\sqrt{\Ilog\lc{}h_{t}\lc\bar{\mu}_{j}\rc\rc{}}
\end{align}

\noindent for all $t\geq 0$ and a.e.~$j\in\mathbb{N}$. Taken together, Equation \ref{EQ.LEM.L2W_Ric_FI_2} and Equation \ref{EQ.LEM.L2W_Ric_FI_3} reduce our claim to the finite-dimensional setting.\par
Assume $A$ and $B$ are finite-dimensional. We show Equation \ref{EQ.LEM.L2W_Ric_FI_1}. Let $\mu,\eta\in\SII(A)$. Using the semigroup property, $1)$ in Corollary \ref{COR.L2W_Log_Mean_NCDS} and $3)$ in Proposition \ref{PRP.QF_AF} let us calculate

\begin{align}\label{EQ.LEM.L2W_Ric_FI_4}
\Ilog\lc{}h_{t}(\mu)\rc{}=-\frac{d}{dt}\Enttau\lc{}h_{t}(\mu)\rc{}=\tau\lc\Delta h_{t}\lc\sharp\mu\rc\log h_{t}\lc\sharp\mu\rc\rc{}
\end{align}

\noindent for all $t>0$. We extend to $t\geq 0$ by continuity. Equation \ref{EQ.LEM.L2W_Ric_FI_4} shows $t\mapsto\sqrt{\Ilog\lc{}h_{t}(\mu)\rc{}}$ is continuous on $[0,\infty)$. Using triangle inequality, we calculate

\begin{align*}
\frac{d^{+}}{dt}\mathcal{W}_{\nabla}^{\log}\lc{}h_{t}(\mu),\eta\rc{} & = \limsup_{s\downarrow 0}\hspace{0.025cm} s^{-1}\lc\mathcal{W}_{\nabla}^{\log}\lc{}h_{t+s}(\mu),\eta\rc{}-\mathcal{W}_{\nabla}^{\log}\lc{}h_{t}(\mu),\eta\rc\rc \phantom{\Bigg)} \\
& \leq \limsup_{s\downarrow 0}\hspace{0.025cm} s^{-1}\mathcal{W}_{\nabla}^{\log}\lc{}h_{t}(\mu),h_{t+s}(\mu)\rc \phantom{\Bigg)}
\end{align*}

\noindent for all $t\geq 0$.\par


\pagebreak


For all $s>0$, we claim

\begin{align}\label{EQ.LEM.L2W_Ric_FI_5}
s^{-1}\mathcal{W}_{\nabla}^{\log}\lc{}h_{t}(\mu),h_{t+s}(\mu)\rc\leq s^{-1}\cdot \int_{t}^{t+s}\sqrt{\Ilog\lc{}h_{r}(\mu)\rc{}}dr
\end{align}

\noindent for all $t\geq 0$. If Equation \ref{EQ.LEM.L2W_Ric_FI_5} holds, then continuity of $t\mapsto\sqrt{\Ilog\lc{}h_{t}(\mu)\rc{}}$ on $[0,\infty)$ and Equation \ref{EQ.LEM.L2W_Ric_FI_5} itself let us calculate

\begin{align}\label{EQ.LEM.L2W_Ric_FI_6}
\frac{d^{+}}{dt}\mathcal{W}_{\nabla}^{\log}\lc{}h_{t}(\mu),\eta\rc\leq\limsup_{s\downarrow 0}\hspace{0.025cm} s^{-1}\int_{t}^{t+s}\sqrt{\Ilog\lc{}h_{t}(\mu)\rc{}}dr=\sqrt{\Ilog\lc{}h_{t}(\mu)\rc{}}
\end{align}

\noindent for all $t\geq 0$. Equation \ref{EQ.LEM.L2W_Ric_FI_6} shows Equation \ref{EQ.LEM.L2W_Ric_FI_1}. We therefore show Equation \ref{EQ.LEM.L2W_Ric_FI_5}. Let $t\geq 0$. For all $s>0$, set $\mu(r):=h_{r}(\mu)$ and $w(r):=-\lc\nabla\sharp\mu(r)\rc^{\flat}$ for all $r\in \lb{}t,t+s\rb$. We show $(\mu,w)\in\Adm^{\lb{}t,t+s\rb{}}\lc{}h_{t}(\mu),h_{t+s}(\mu)\rc$ in the proof of Corollary \ref{COR.L2W_Log_Mean_NCDS}. Let $L^{\log}$ denote the length functional in our case. Using scale invariance of length functionals as per Proposition \ref{PRP.Length_Functional_Reparametrisation}, we directly verify

\begin{align}\label{EQ.LEM.L2W_Ric_FI_7}
L^{\log}(\mu,w)=\int_{t}^{t+s}\sqrt{\Ilog\lc{}h_{t}(\mu)\rc{}}dr.
\end{align}

\noindent Equation \ref{EQ.LEM.L2W_Ric_FI_7} shows Equation \ref{EQ.LEM.L2W_Ric_FI_5}. The general case follows as discussed above.
\end{proof}

\begin{thm}\label{THM.L2W_Ric_FI}
Let $(\phi,\bpsi,\gamma,\nabla)$ be noncommutative differential structure for tracial AF-$C^{*}$-algebras $(A,\tau)$ and $(B,\omega)$ in the logarithmic mean setting.

\begin{itemize}
\item[1)] If $\Ric\nabla\geq\lambda$, then $\Enttau$ satisfies $\HWI_{\lambda}$.

\item[2)] If $\Enttau$ satisfies $\HWI_{\lambda}$ for $\lambda>0$, then $\Enttau$ satisfies $\MLSI_{\lambda}$.

\item[3)] If $\Enttau$ satisfies $\MLSI_{\lambda}$, then $\Enttau$ satisfies $\TW_{\lambda}$.
\end{itemize}
\end{thm}
\begin{proof}
We adapt the proofs of Theorem 11.3, Theorem 11.4 and Theorem 11.5 in \cite{ART.Car_Maa.2020.Quantum_OT_III} to the AF-$C^{*}$-setting by means of the coarse graining process. Non-ergodicity requires us to consider relative entropy of finitely supported fixed states.\par
We reduce to the finite-dimensional setting. Let $\xi\in\SII(A)$ be a finitely supported fixed state. Let $\mathcal{C}\subset (\SII(A),\mathcal{W}_{\nabla}^{\log})$ be finitely supported with fixed part $\xi$ s.t.~we have $\mathcal{C}\cap\dom\Enttau\neq\emptyset$. For all $j\in\mathbb{N}$ s.t.~$\xi_{j}\neq 0$, Equation \ref{EQ.SSEC.L2W_EVI_Equivalence_2} and Equation \ref{EQ.THM.L2W_EVI_Equivalence_1} together with $2.1)$ in Proposition \ref{PRP.QF_AF} show $\HWI_{\lambda}$ restricts to

\begin{align}\label{EQ.THM.L2W_Ric_FI_1}
\Ent\lc\bar{\mu}_{j},\tau\rc\leq\mathcal{W}_{\nabla}^{\log}\lc\bar{\mu}_{j},\bar{\xi}_{j}\rc\sqrt{\Ilog\lc\bar{\mu}_{j}\rc{}}-\frac{\lambda}{2}\mathcal{W}_{\nabla}^{\log}\lc\bar{\mu}_{j},\bar{\xi}_{j}\rc^{2}+\Ent\lc\bar{\xi}_{j},\tau\rc{} \tag{$\EVI_{\lambda}^{j}$}
\end{align}


\pagebreak


\noindent on $\mathcal{C}_{A_{j}}\lc\bar{\xi}_{j}\rc$ for all $\lambda\in\mathbb{R}$, resp.~$\MLSI_{\lambda}$ and $\TW_{\lambda}$ restrict to

\begin{align}\label{EQ.THM.L2W_Ric_FI_2}
\Ent\lc\bar{\mu}_{j},\tau\rc{} & \leq \frac{1}{2\lambda}\Ilog\lc\bar{\mu}_{j}\rc{}+\Ent\lc\bar{\xi}_{j},\tau\rc{} \tag{$\MLSI_{\lambda}^{j}$}
\end{align}

\noindent and

\begin{align}\label{EQ.THM.L2W_Ric_FI_3}
\mathcal{W}_{\nabla}^{\log}\lc\bar{\mu}_{j},\bar{\xi}_{j}\rc{} & \leq \sqrt{\frac{2}{\lambda}\bigg(\hspace{-0.028975cm} \Ent\lc\bar{\mu}_{j},\tau\rc{}-\Ent\lc\bar{\xi}_{j},\tau\rc\bigg)} \tag{$\TW_{\lambda}^{j}$}
\end{align}

\noindent on $\mathcal{C}_{A_{j}}\lc\bar{\xi}_{j}\rc$ for all $\lambda>0$. If conversely $\HWI_{\lambda}^{j}$, $\MLSI_{\lambda}^{j}$, resp.~$\TW_{\lambda}^{j}$ holds for a.e.~$j\in\mathbb{N}$, then note Equation \ref{EQ.SSEC.L2W_EVI_Equivalence_6} and Equation \ref{EQ.SSEC.L2W_EVI_Equivalence_8} together with $2.2)$ in Proposition \ref{PRP.QF_AF} show letting $j\uparrow\infty$ on both sides of the given inequality recovers $\HWI_{\lambda}$, $\MLSI_{\lambda}$, resp.~$\TW_{\lambda}$ on $\mathcal{C}\cap\dom\Enttau$. We therefore reduce to the finite-dimensional setting.\par
Assume $A$ and $B$ are finite-dimensional. Let $\mu\in\mathcal{C}$. We show $1)$. Assume $\Ric\nabla\geq\lambda$. If $\Ilog(\mu)=\infty$, then there is nothing to show. Assume $\Ilog(\mu)<\infty$. Theorem \ref{THM.L2W_EVI_Equivalence} shows $\EVI_{\lambda}$ as per Equation \ref{EQ.SSEC.L2W_EVI_Equivalence_3} applies. Corollary \ref{COR.QOT_Distance_AC_Rel_Ent} shows $\mu,\xi\in\mathcal{C}_{A}(\xi)$. We obtain

\begin{align}\label{EQ.THM.L2W_Ric_FI_4}
\Ent(\mu,\tau)\leq-\frac{1}{2}\restr{0.925}{\frac{d^{+}}{dt}}{t=0}\mathcal{W}_{\nabla}^{\log}\lc{}h_{t}(\mu),\xi\rc^{2}-\frac{\lambda}{2}\mathcal{W}_{\nabla}^{\log}(\mu,\xi)^{2}+\Ent(\xi,\tau)
\end{align}

\noindent for $t=0$ by rearranging terms accordingly. Equation \ref{EQ.THM.L2W_Ric_FI_4} shows $\Enttau$ satisfies $\HWI_{\lambda}$ if

\begin{align}\label{EQ.THM.L2W_Ric_FI_5}
-\frac{1}{2}\restr{0.925}{\frac{d^{+}}{dt}}{t=0}\mathcal{W}_{\nabla}^{\log}\lc{}h_{t}(\mu),\xi\rc^{2}\leq\mathcal{W}_{\nabla}^{\log}(\mu,\xi)\sqrt{\Ilog(\mu)}.
\end{align}

\noindent We show Equation \ref{EQ.THM.L2W_Ric_FI_5}. Note $2)$ in Corollary \ref{COR.QOT_Distance_AC_L2} shows $\mathcal{W}_{\nabla\vert\mathcal{C}_{A}(\xi)\times\mathcal{C}_{A}(\xi)}^{\log}$ is finite and $\|.\|_{A}$-continuous. Using the latter, we have

\begin{align}\label{EQ.THM.L2W_Ric_FI_6}
\limsup_{t\downarrow 0}\hspace{0.025cm} \mathcal{W}_{\nabla}^{\log}\lc{}h_{t}(\mu),\mu\rc{}=0,\ \limsup_{t\downarrow 0}\hspace{0.025cm} \mathcal{W}_{\nabla}^{\log}\lc{}h_{t}(\mu),\xi\rc{}=\mathcal{W}_{\nabla}^{\log}(\mu,\xi).
\end{align}

\noindent Using triangle inequality, symmetry of distances and Equation \ref{EQ.THM.L2W_Ric_FI_6} let us calculate

\begingroup
\allowdisplaybreaks
\begin{align*}
& \ -\frac{1}{2}\restr{0.925}{\frac{d^{+}}{dt}}{t=0}\mathcal{W}_{\nabla}^{\log}\lc{}h_{t}(\mu),\xi\rc^{2} \phantom{\Bigg)} \\
=& \ \limsup_{t\downarrow 0}\hspace{0.025cm} \frac{1}{2}t^{-1}\lc\mathcal{W}_{\nabla}^{\log}(\mu,\xi)^{2}-\mathcal{W}_{\nabla}^{\log}\lc{}h_{t}(\mu),\xi\rc^{2}\rc \phantom{\Bigg)} \\
\leq& \ \limsup_{t\downarrow 0}\hspace{0.025cm} \frac{1}{2}t^{-1}\lc\lc\mathcal{W}_{\nabla}^{\log}\lc{}h_{t}(\mu),\mu\rc{}+\mathcal{W}_{\nabla}^{\log}\lc{}h_{t}(\mu),\xi\rc\rc^{2}-\mathcal{W}_{\nabla}^{\log}\lc{}h_{t}(\mu),\xi\rc^{2}\rc \phantom{\Bigg)} \\
=& \ \limsup_{t\downarrow 0}\hspace{0.025cm} \frac{1}{2}t^{-1}\lc\mathcal{W}_{\nabla}^{\log}\lc{}h_{t}(\mu),\mu\rc^{2}+2\mathcal{W}_{\nabla}^{\log}\lc{}h_{t}(\mu),\mu\rc\mathcal{W}_{\nabla}^{\log}\lc{}h_{t}(\mu),\xi\rc\rc \phantom{\Bigg)} \\
=& \ \limsup_{t\downarrow 0}\hspace{0.025cm} \frac{1}{2}t^{-1}\mathcal{W}_{\nabla}^{\log}\lc{}h_{t}(\mu),\mu\rc^{2} + \lc\limsup_{t\downarrow 0}\hspace{0.025cm} t^{-1}\mathcal{W}_{\nabla}^{\log}\lc{}h_{t}(\mu),\mu\rc\rc\cdot \mathcal{W}_{\nabla}^{\log}(\mu,\xi) \phantom{\Bigg)} \\
=& \ 0 + \lc\limsup_{t\downarrow 0}\hspace{0.025cm} t^{-1}\mathcal{W}_{\nabla}^{\log}\lc{}h_{t}(\mu),\mu\rc\rc\cdot \mathcal{W}_{\nabla}^{\log}(\mu,\xi) \phantom{\Bigg)} \\
=& \ \mathcal{W}_{\nabla}^{\log}(\mu,\xi)\cdot \lc\limsup_{t\downarrow 0}\hspace{0.025cm} t^{-1}\lc\mathcal{W}_{\nabla}^{\log}\lc{}h_{t}(\mu),\mu\rc{}-\mathcal{W}_{\nabla}^{\log}(\mu,\mu)\rc\rc{} \phantom{\Bigg)} \\
=& \ \mathcal{W}_{\nabla}^{\log}(\mu,\xi)\cdot \restr{0.925}{\frac{d^{+}}{dt}}{t=0}\mathcal{W}_{\nabla}^{\log}\lc{}h_{t}(\mu),\mu\rc{}. \phantom{\Bigg)}
\end{align*}
\endgroup

\noindent Applying Lemma \ref{LEM.L2W_Ric_FI} to the final term in the above calculation yields Equation \ref{EQ.THM.L2W_Ric_FI_5} as required. As such, we know $\Enttau$ satisfies $\HWI_{\lambda}$. Get $1)$.\par
We show $2)$. Assume $\Enttau$ satisfies $\HWI_{\lambda}$ for $\lambda>0$. Note Young's inequality implies $xy\leq Cx^{2}+(4C)^{-1}y^{2}$ for all $x,y\in\mathbb{R}$ and $C>0$ \cite{ART.Erb_Maa.2012.Discrete_OT_Ricci_Bounds}. Using $C=2^{-1}\lambda$, we obtain

\begin{align}\label{EQ.THM.L2W_Ric_FI_7}
\mathcal{W}_{\nabla}^{\log}(\mu,\xi)\sqrt{\Ilog(\mu)}-\frac{\lambda}{2}\mathcal{W}_{\nabla}^{\log}(\mu,\xi)^{2}\leq\frac{1}{2\lambda}\Ilog(\mu)
\end{align}

\noindent by rearranging terms accordingly. $\HWI_{\lambda}$ and Equation \ref{EQ.THM.L2W_Ric_FI_7} let us calculate

\begin{align}\label{EQ.THM.L2W_Ric_FI_8}
\Ent(\mu,\tau)\leq\mathcal{W}_{\nabla}^{\log}(\mu,\xi)\sqrt{\Ilog(\mu)}-\frac{\lambda}{2}\mathcal{W}_{\nabla}^{\log}(\mu,\xi)^{2}+\Ent(\xi,\tau)\leq\frac{1}{2\lambda}\Ilog(\mu)+\Ent(\xi,\tau).
\end{align}

\noindent Equation \ref{EQ.THM.L2W_Ric_FI_8} shows $\Enttau$ satisfies $\MLSI_{\lambda}$. Get $2)$.\par
We show $3)$. Assume $\Enttau$ satisfies $\MLSI_{\lambda}$. We know $\lambda>0$ by hypothesis. Set

\begin{align}\label{EQ.THM.L2W_Ric_FI_9}
F(t):=\mathcal{W}_{\nabla}^{\log}\lc\mu,h_{t}(\mu)\rc{}+\sqrt{\frac{2}{\lambda}\bigg(\hspace{-0.028975cm} \Ent\lc{}h_{t}(\mu),\tau\rc{}-\Ent(\xi,\tau)\bigg)}
\end{align}

\noindent for all $t\geq 0$. Using $1)$ in Theorem \ref{THM.Wstar_Derivation_QG_HSG_Regularity} and $3)$ in Proposition \ref{PRP.QF_AF}, we directly verify Equation \ref{EQ.THM.L2W_Ric_FI_9} defines continuous map $F:(0,\infty)\longrightarrow\mathbb{R}$ s.t.~$\frac{d^{+}}{dt}F$ exists for all $t\geq 0$. Norm continuity and Theorem \ref{THM.L2W_Log_Mean_NCDS} imply

\begin{align}\label{EQ.THM.L2W_Ric_FI_10}
F(0):=\lim_{t\downarrow 0}\hspace{0.025cm} F(t)=\sqrt{\frac{2}{\lambda}\lc\Ent(\mu,\tau)-\Ent(\xi,\tau)\rc{}},\ F\lc\infty\rc{}:=\lim_{t\uparrow\infty}\hspace{0.025cm} F(t)=\mathcal{W}_{\nabla}^{\log}(\mu,\xi).
\end{align}

\noindent Integrating over $[0,\infty)$, Equation \ref{EQ.THM.L2W_Ric_FI_10} implies $\Enttau$ satisfies $\TW_{\lambda}$ if $\frac{d^{+}}{dt}F(t)\leq 0$ for all $t>0$. We show the latter.\par


\pagebreak


Using the semigroup property, $3)$ in Proposition \ref{PRP.QF_AF} and $\MLSI_{\lambda}$ let us calculate

\begin{align*}
\frac{d}{dt}\sqrt{\frac{2}{\lambda}\cdot \bigg(\hspace{-0.028975cm} \Ent\lc{}h_{t}(\mu),\tau\rc{}-\Ent(\xi,\tau)\bigg)}& =-\frac{\Ilog\lc{}h_{t}(\mu)\rc{}}{\sqrt{2\lambda\cdot \bigg(\hspace{-0.028975cm} \Ent\lc{}h_{t}(\mu),\tau\rc{}-\Ent(\xi,\tau)\bigg)}} \phantom{\bigg)} \\
& \\
& \leq-\sqrt{\Ilog\lc{}h_{t}(\mu)\rc{}} \phantom{\bigg)}
\end{align*}

\noindent in each case. Note we use $\MLSI_{\lambda}$ in the denominator. Applying Lemma \ref{LEM.L2W_Ric_FI} and the above calculation to Equation \ref{EQ.THM.L2W_Ric_FI_9} shows

\begin{align}\label{EQ.THM.L2W_Ric_FI_11}
\frac{d^{+}}{dt}F(t)\leq\sqrt{\Ilog\lc{}h_{t}(\mu)\rc{}}-\frac{d}{dt}\sqrt{\frac{2}{\lambda}\bigg(\hspace{-0.028975cm} \Ent\lc{}h_{t}(\mu),\tau\rc{}-\Ent(\xi,\tau)\bigg)}\leq 0
\end{align}

\noindent for all $t>0$. Equation \ref{EQ.THM.L2W_Ric_FI_11} shows $\frac{d^{+}}{dt}F(t)\leq 0$ for all $t>0$ as required. As such, we know $\Enttau$ satisfies $\TW_{\lambda}$. Get $3)$.
\end{proof}

\begin{cor}\label{COR.L2W_Ric_FI}
If $\Ric\nabla\geq\lambda>0$, then $\Enttau$ satisfies $\HWI_{\lambda}$, $\MLSI_{\lambda}$ and $\TW_{\lambda}$.
\end{cor}
\begin{proof}
Apply Theorem \ref{THM.L2W_Ric_FI}.
\end{proof}


\begin{appendices}


\chapter{Operator Theory}\label{APP.A}

\vspace{-0.15cm}
We review operator theory. In Section \ref{SEC.A_Fnd}, we cover fundamental results for unbounded operators, $C^{*}$-~and $W^{*}$-algebras, as well as functional calculus used in our discussion. In Section \ref{SEC.A_Maps}, we discuss strong resolvent convergence, resolvent-preserving maps of unbounded operators, and introduce compression maps.


\vspace{-0.01675cm}
\section{Fundamental operator theory}\label{SEC.A_Fnd}

\vspace{-0.01675cm}
In Subsection \ref{SSEC.A_Fnd_Unbd}, we review partial orders generated by positive elements, as well as spaces of bounded and unbounded operators on Hilbert spaces. We further discuss twisting maps on spaces of unbounded operators induced by Hilbert space isometries. In Subsection \ref{SSEC.A_Fnd_CWstar}, $C^{*}$-~and $W^{*}$-algebras are covered. We give direct sums and tensor products. We discuss normal, completely positive and completely Markovian maps.\par
In Subsection \ref{SSEC.A_Fnd_FC}, we review functional calculus. Spectral measures of self-adjoint unbounded operators are given by the well-established bounded measurable functional calculus for $W^{*}$-algebras. Joint spectral measures are given by tensoring such spectral measures of strongly commuting self-adjoint unbounded operators. Functional calculus is integration w.r.t.~spectral measures. Joint functional calculus is integration w.r.t.~joint spectral measures. We introduce two related standard operations for further use in our discussion. In Lemma \ref{LEM.JFC_Preservation}, we establish pull-back along tensor products of normal unital $^{*}$-homomorphisms. In Subsection \ref{SSEC.A_Maps_Compression}, we study compression.


\vspace{-0.01675cm}
\subsection{Unbounded operators}\label{SSEC.A_Fnd_Unbd}

\vspace{-0.01675cm}
Standard references for unbounded operators are \cite{BK.Ped.1989.Analysis_Now}, \cite{BK.Sch.2012.Unbounded_Operators} and \cite{BK.Tak.1979.OpAlg_I}.


\vspace{-0.01675cm}
\subsubsection*{Partial orders generated by positive elements}

We use $\mathbb{K}\in\lset\mathbb{C},\mathbb{R}\rset$ as field.

\vspace{-0.01675cm}
\begin{dfn}\label{DFN.PO_I}
Let $V$ be a complex vector space. A convex cone $C\subset V$ is proper if $0\in C$ and $C\cap -C=\lset{}0\rset$. Let $\gamma:V\longrightarrow V$ be anti-linear involution. Its set of hermitian elements is $V_{h}:=\lset{}v\in V\ \vset\ \gamma(v)=v\rset$.

\vspace{-0.01675cm}
\begin{itemize}
\item[1)] For all $v\in V$, set $\RE(v):=\frac{1}{2}\lc{}v+\gamma(v)\rc$ and $\IM(v):=\frac{1}{2i}\lc{}v-\gamma(v)\rc$.

\item[2)] If $V_{h}$ has partial order, then we call it generated by its set $V_{+}:=\lset{}v\in V_{h}\ \vset\ v\geq 0\rset$ of positive elements if $V_{+}$ is a proper cone generating the partial order.
\end{itemize}
\end{dfn}


\pagebreak


\begin{ntn}\label{NTN.PO}
We use vector spaces $V_{s}$ with subscript $s\in S$ in an index set. If $V_{s}$ has partial order generated by positive elements, then we write $V_{s,h}$ and $V_{s,+}$ to denote its set of hermitian, resp.~positive elements.
\end{ntn}

\begin{rem}
If $V_{h}$ has partial order generated by its set $V_{+}$ of positive elements, then $V_{h}=\langle V_{+}\rangle_{\mathbb{R}}\oplus\langle V_{+}\rangle_{\mathbb{R}}$ using direct sum of real vector spaces. Since further $V=\langle V_{h}\rangle_{\mathbb{R}}\oplus\langle V_{h}\rangle_{\mathbb{R}}$ using decomposition as per Equation \ref{EQ.PRP.PO_I_1}, we say that $V$ has partial order generated by its set $V_{+}$ of positive elements in this case.
\end{rem}

\begin{prp}\label{PRP.PO_I}
Let $V$ be a complex vector space. We consider anti-linear involution $\gamma:V\longrightarrow V$. For all $v\in V$, we have $\RE(v),\IM(v)\in V_{h}$ and

\begin{align}\label{EQ.PRP.PO_I_1}
v=\RE(v)+i\IM(v),\ \gamma(v)=\RE(v)-i\IM(v).
\end{align}
\end{prp}
\begin{proof}
Apply anti-linearity of $\gamma$.
\end{proof}

\begin{dfn}\label{DFN.PO_II}
Let $V$ and $W$ be complex vector spaces. We consider anti-linear\linebreak involutions $\gamma^{V}:V\longrightarrow V$ and $\gamma^{W}:W\longrightarrow W$. Let $\phi:V\longrightarrow W$ be a linear map.

\begin{itemize}
\item[1)] We call $\phi$ order-preserving if

\begin{itemize}
\item[1.1)] $\phi\lc{}V_{h}\rc\subset W_{h}$,

\item[1.2)] $v_{1}\leq v_{2}$ in $V_{h}$ implies $\phi\lc{}v_{1}\rc\leq\phi\lc{}v_{2}\rc$ in $W_{h}$.
\end{itemize}

\begin{reapply}
\end{reapply}

\item[2)] Assume $V_{h}$ and $W_{h}$ have partial orders generated by positive elements. We call $\phi$ positivity-preserving if $\phi\lc{}V_{+}\rc\subset W_{+}$.
\end{itemize}
\end{dfn}

\begin{prp}\label{PRP.PO_II}
Let $V$ and $W$ be complex vector spaces. We consider anti-linear\linebreak involutions $\gamma^{V}:V\longrightarrow V$ and $\gamma^{W}:W\longrightarrow W$. Let $\phi:V\longrightarrow W$ be a linear map.

\begin{itemize}
\item[1)] $\phi$ is order-preserving if and only if $\phi\circ\gamma^{V}=\gamma^{W}\circ\phi$.

\item[2)] If $V_{h}$ and $W_{h}$ have partial orders generated by positive elements, then $\phi$ is order-preserving if and only if $\phi$ is positivity-preserving.
\end{itemize}
\end{prp}
\begin{proof}
We have $\phi\lc{}V_{h}\rc\subset W_{h}$ if and only if $\phi\lc\RE(v)\rc{}=\RE\lc\phi(v)\rc$ and $\phi\lc\IM(v)\rc{}=\IM\lc\phi(v)\rc$ for all $v\in V$. This implies $1)$. Get $2)$ since $V_{h}=\langle V_{+}\rangle_{\mathbb{R}}\oplus\langle V_{+}\rangle_{\mathbb{R}}$ and $W_{h}=\langle W_{+}\rangle_{\mathbb{R}}\oplus\langle W_{+}\rangle_{\mathbb{R}}$ generate the respective partial orders.
\end{proof}


\subsubsection*{Bounded and unbounded operators}

For spaces of bounded operators, we fix notation. This includes operator topologies used throughout our discussion. For spaces of unbounded operators, we fix notation, set partial order in Definition \ref{DFN.Unbd_PO}, and give twisting maps of Hilbert space isometries in Definition \ref{DFN.Unbd_Twist}. We collect properties of such twisting maps in Proposition \ref{PRP.Unbd_Twist}.\par


\pagebreak


\begin{dfn}
Let $\lc{}V,\|.\|_{V}\rc$ and $\lc{}W,\|.\|_{W}\rc$ be Banach spaces.

\begin{itemize}
\item[1)] Let $\BII\lc{}V,W\rc$ be the set of all $\lc\|.\|_{V},\|.\|_{W}\rc$-bounded operators and let $\|.\|_{\BII\lc{}V,W\rc{}}$ be its operator norm.

\item[2)] Set $\BII(V):=\BII\lc{}V,V\rc$ and $I_{V}:=\id_{\BII(V)}$. We call $V^{*}:=\BII\lc{}V,\mathbb{C}\rc$ Banach dual of $V$.
\end{itemize}
\end{dfn}

\begin{ntn}
Unless stated otherwise, we suppress Banach space norms.
\end{ntn}

Operator norms determine uniform operator topology, also called norm topology. Let $H$ be a Hilbert space. We equip $\BII(H)$ with several other operator topologies aside from the uniform one: the $\sigma$-strong and $\sigma$-weak, as well as the strong and weak operator topology. For details on these operator topologies, we refer to Chapter II.2 in \cite{BK.Tak.1979.OpAlg_I}.

\begin{ntn}
For a normed vector space $\lc{}V,\|.\|_{V}\rc$, let $v=\|.\|_{V}$-$\lim_{k\in K}v_{k}$ denote norm convergence of nets in $V$ and $u=w^{*}$-$\lim_{k\in K}u_{k}$ denote $w^{*}$-convergence of nets in $V^{*}$. For a Hilbert space $\lc{}H,\|.\|_{H}\rc$, let $x=\s$-$\lim_{k\in K}x_{k}$ denote strong and $x=\w$-$\lim_{k\in K}x_{k}$ denote weak convergence of nets in $\BII(H)$.
\end{ntn}

\begin{rem}\label{REM.Wstar_Con}
The $\sigma$-strong and strong topologies are equivalent on norm bounded sets \lc{}cf.~Lemma II.2.5 in \cite{BK.Tak.1979.OpAlg_I}\rc{}. Equally, the $\sigma$-weak and weak topologies are.
\end{rem}

For details on elementary unbounded operator theory, we refer to \cite{BK.Ped.1989.Analysis_Now}.

\begin{dfn}\label{DFN.Unbd_PO}
Let $H$ be a Hilbert space, $\UBII(H)$ the set of all unbounded operators on $H$, and $\UBII(H)_{h}$ the set of all self-adjoint unbounded operators on $H$.

\begin{itemize}
\item[1)] For all $T,S\in\UBII(H)_{h}$, set $T\geq S$ if and only if

\begin{itemize}
\item[1.1)] $\dom T\subset\dom S$, \phantom{\big)}

\item[1.2)] $\lgl T(u),u\rgl_{H}\geq\lgl S(u),u\rgl_{H}$ for all $u\in\dom T$. \phantom{\big)}
\end{itemize}

\begin{reapply}
\end{reapply}

\item[2)] We call $T\in\UBII(H)_{h}$ positive if $\lgl T(u),u\rgl_{H}\geq 0$ for all $u\in\dom T$. Let $\UBII(H)_{+}$ be the set of all positive unbounded operators on $H$.
\end{itemize}
\end{dfn}

\begin{rem}\label{REM.Unbd_PO}
We equip $\UBII(H)$ with canonical addition and scalar multiplication \lc{}cf.~Chapter 5 in \cite{BK.Ped.1989.Analysis_Now}\rc{}. We obtain complex unital semi-module $\UBII(H)$ satisfying all vector space axioms except additive inverses. Linear maps, inclusions and proper cones are defined as for complex vector spaces. Functional calculus shows $\UBII(H)_{+}\subset\UBII(H)_{h}$ is a proper cone generating partial order defined as per $1)$ in Definition \ref{DFN.Unbd_PO}.
\end{rem}

\begin{dfn}\label{DFN.Unbd_Twist}
Let $H_{0}$ and $H_{1}$ be Hilbert spaces. Let $\phi:H_{0}\longrightarrow H_{1}$ be a linear or anti-linear isometric isomorphism. For all $T\in\UBII\lc{}H_{0}\rc$, we define $\phi^{\dagger}(T)\in\UBII\lc{}H_{1}\rc$ as

\begin{itemize}
\item[1)] $\dom\phi^{\dagger}(T):=\big\{\hspace{0.025cm} u\in H_{1}\ \vset\ \phi^{-1}(u)\in\dom T\hspace{0.025cm} \big\}$, \phantom{\big)}

\item[2)] $\phi^{\dagger}(T)(u):=\phi\lc{}T\lc\phi^{-1}(u)\rc\rc$ for all $u\in\dom\phi^{\dagger}(T)$. \phantom{\big)}
\end{itemize}

\noindent This defines map $\phi^{\dagger}:\UBII\lc{}H_{0}\rc\longrightarrow\UBII\lc{}H_{1}\rc$ by $T\mapsto\phi^{\dagger}(T)$. Using $\phi^{-1}:H_{1}\longrightarrow H_{0}$ defines map $\phi^{-\dagger}:\UBII\lc{}H_{1}\rc\longrightarrow\UBII\lc{}H_{0}\rc$ by $T\mapsto\phi^{-\dagger}(T):=\lc\phi^{-1}\rc^{\dagger}(T)$.
\end{dfn}


\pagebreak


Proposition \ref{PRP.Unbd_Twist} shows twisting preserves standard operations for densely defined closable unbounded operators if they are defined in the domain.

\begin{prp}\label{PRP.Unbd_Twist}
Let $\phi:H_{0}\longrightarrow H_{1}$ be a linear or anti-linear isometric isomorphism of Hilbert spaces.

\begin{itemize}
\item[1)] We have bijective linear maps $\phi^{\dagger}$ and $\phi^{-\dagger}=\lc\phi^{-1}\rc^{\dagger}=\lc\phi^{\dagger}\rc^{-1}$.

\item[2)] If $T\in\UBII\lc{}H_{0}\rc$ is densely defined closable, then $\phi^{\dagger}(T)\in\UBII\lc{}H_{1}\rc$ is.

\item[3)] Let $T,S\in\UBII\lc{}H_{0}\rc$ be densely defined closable s.t.~$T+S$ and $TS$ are densely defined closable. For all $\lambda_{0},\lambda_{1}\in\mathbb{C}$, we have

\begin{itemize}
\item[3.1)] $\phi^{\dagger}\lc{}T^{*}\rc{}=\phi^{\dagger}(T)^{*}$, \phantom{\vstretch{0.75}{\bigg)}}

\item[3.2)] $\phi^{\dagger}\lc\overline{\lambda_{0}T+\lambda_{1}S}\rc{}=\overline{\lambda_{0}\phi^{\dagger}(T)+\lambda_{1}\phi^{\dagger}(S)}$, \phantom{\vstretch{0.75}{\bigg)}}

\item[3.3)] $\phi^{\dagger}\lc\overline{TS}\rc{}=\overline{\phi^{\dagger}(T)\phi^{\dagger}(S)}$. \phantom{\vstretch{0.75}{\bigg)}}
\end{itemize}

\begin{reapply}
\end{reapply}

\end{itemize}
\end{prp}
\begin{proof}
Since $\phi$ is anti-linear if and only if $\phi^{-1}$ is, we assume $\phi$ is linear without loss of generality. Get $1)$ by construction. Let $T\in\UBII\lc{}H_{0}\rc$. Written as graph, $\phi^{\dagger}$ maps $T$ to

\begin{align}\label{EQ.PRP.Unbd_Twist_1}
\phi^{\dagger}(T)=\lset\lc\phi(u),\phi\lc{}T(u)\rc\rc\in\phi\lc\dom T\rc\times H_{1}\ \vset\ u\in\dom T\rset{}.
\end{align}

\noindent Equation \ref{EQ.PRP.Unbd_Twist_1} shows $2)$ and $3)$ because $\phi$ is isometric isomorphism of Hilbert spaces.
\end{proof}


\subsection[$C^{*}$-~and $W^{*}$-algebras]{$\mathbf{C}^{*}$-~and $\mathbf{W}^{*}$-algebras}\label{SSEC.A_Fnd_CWstar}

Standard references for the theory of $C^{*}$-~and $W^{*}$-algebras are \cite{BK.Bla.2006.OpAlg} and \cite{BK.Tak.1979.OpAlg_I}\cite{BK.Tak.2003.OpAlg_II}\cite{BK.Tak.2003.OpAlg_III}. We use \cite{BK.Kad_Rin.1997.OpAlg_I}\cite{BK.Kad_Rin.1997.OpAlg_II} and \cite{BK.Ped.2018.Cstar_Algebras} as supplement. Moreover, \cite{BK.Bro_Oza.2008.Cstar_AF} focuses on the approximately finite-dimensional, or AF-$C^{*}$-setting, and \cite{BK.Dav.1996.Cstar_By_Example} is a source of examples.


\subsubsection*{$\mathbf{C}^{*}$-algebras}

The $C^{*}$-identity, i.e.~Equation \ref{EQ.DFN.Cstar_3}, defines $C^{*}$-algebras. It imposes a rigid structure on such Banach $^{*}$-algebras. All $^{*}$-homomorphisms of $C^{*}$-algebras are bounded of norm at most one, and isometries if injective. Thus all $^{*}$-isomorphisms of $C^{*}$-algebras are isometries, hence a $^{*}$-algebra has at most one $C^{*}$-norm.

\begin{dfn}\label{DFN.Cstar}
Let $\lc{}A,\|.\|_{A}\rc$ be a Banach $^{*}$-algebra. It is unital if it has unit $1_{A}\in A$.

\begin{itemize}
\item[1)] The hermitian, resp.~positive elements in $A$ are

\begin{align}\label{EQ.DFN.Cstar_1}
A_{h}:=\lset{}x\in A\ \vset\ x=x^{*}\rset{},\ A_{+}:=\lset{}x\in A_{h}\ \vset\ \exists y\in A:\ x=y^{*}y\rset{}.
\end{align}

\begin{reapply}
\end{reapply}

\item[2)] The hermitian, resp.~positive bounded functionals on $A$ are

\begin{align}\label{EQ.DFN.Cstar_2}
A_{h}^{*}:=\lset\mu\in A^{*}\ \vset\ \forall x\in A:\ \mu(x^{*}x)\in\mathbb{R}\rset{},\ A_{+}^{*}:=\lset\mu\in A_{h}^{*}\ \vset\ \forall x\in A:\ \mu(x^{*}x)\geq 0\rset{}.
\end{align}

\begin{reapply}
\end{reapply}


\pagebreak


\item[3)] $A$ is a $C^{*}$-algebra if all $x\in A$ satisfy the $C^{*}$-identity 

\begin{align}\label{EQ.DFN.Cstar_3}
\|x^{*}x\|_{A}=\|x\|_{A}^{2}.
\end{align}

\begin{reapply}
\end{reapply}

\end{itemize}
\end{dfn}

\begin{ntn}
Let $A$ be a $C^{*}$-algebra. Rather than hermitian, we say that $x\in A_{h}$ is self-adjoint and call $\mu\in A_{h}^{*}$ real.
\end{ntn}

\begin{bsp}\label{BSP.Cstar_Quantum}
For all Hilbert spaces $H$, its space $\lc\BII(H),\|.\|_{\BII(H)}\rc$ of bounded and its space $\lc\KII(H),\|.\|_{\BII(H)}\rc$ of compact operators are $C^{*}$-algebras.
\end{bsp}

\begin{bsp}\label{BSP.Cstar_Commutative}
Let $X$ be a locally compact Hausdorff space. Let $C_{0}(X)$ be the set of all continuous $g:X\longrightarrow\mathbb{C}$ vanishing at infinity. Pointwise operations equip it with Banach $^{*}$-algebra structure. For all $g\in C_{0}(X)$, set $\|g\|_{\infty}:=\sup_{x\in X}\absv{1.15}{g(x)}$. We see $C_{0}(X)$ equipped with $\|.\|_{\infty}$ is a $C^{*}$-algebra. Up to natural isomorphisms \cite{BK.MacLane.1998.Category_Theory}, all commutative $C^{*}$-algebras are of such form by Gelfand duality \lc{}cf.~Theorem I.3.11 in \cite{BK.Tak.1979.OpAlg_I}\rc{}. If $X$ is compact, then $C_{0}(X)=C(X)$. Note $C_{b}(X)$ equipped with $\|.\|_{\infty}$ is a $C^{*}$-algebra.
\end{bsp}

We use standard definitions for $^{*}$-algebras. Homomorphisms of $^{*}$-algebras are called $^{*}$-homomorphisms. For $C^{*}$-algebras, Proposition \ref{PRP.Cstar_Hom} shows boundedness follows by the $C^{*}$-identity if $^{*}$-algebra structures are preserved. This leads to Definition \ref{DFN.Cstar_Hom}.

\begin{dfn}\label{DFN.Cstar_Hom}
Let $A$ and $B$ be $C^{*}$-algebras.

\begin{itemize}
\item[1)]  A $^{*}$-homomorphism $\phi:A\longrightarrow B$ of $C^{*}$-algebras is a $^{*}$-homomorphism. If $A$ and $B$ are unital, then $\phi:A\longrightarrow B$ is unital if $\phi(1_{A})=1_{B}$.

\item[2)] If $A\subset B$, then $A$ is a $C^{*}$-subalgebra of $B$ if $A\subset B$ is a $^{*}$-homomorphism. If $A$ and $B$ are furthermore unital, then $A$ is a unital $C^{*}$-subalgebra of $B$ if $A\subset B$ is a unital $^{*}$-homomorphism.
\end{itemize}
\end{dfn}

\begin{bsp}
Let $A$ be a $C^{*}$-algebra and $H$ a Hilbert space. We call $\pi:A\longrightarrow\BII(H)$ a $^{*}$-representation of $A$ over $H$ if it is a $^{*}$-homomorphism. It is faithful if injective. It is unital if $A$ is unital and $\pi(1_{A})=I_{H}$, i.e.~if it is a unital $^{*}$-homomorphism.
\end{bsp}

\begin{prp}\label{PRP.Cstar_Hom}
Let $\phi:A\longrightarrow B$ be a $^{*}$-homomorphism of $C^{*}$-algebras.

\begin{itemize}
\item[1)] $\phi\in\BII\lc{}A,B\rc$ and $\|\phi\|_{\BII\lc{}A,B\rc{}}\leq 1$.

\item[2)] If $\phi$ is injective, then it is an isometry.

\item[3)] If $\phi$ is a $^{*}$-isomorphism, then $\phi^{-1}$ is a $^{*}$-isomorphism.
\end{itemize}
\end{prp}
\begin{proof}
Proposition I.5.2 and Proposition I.5.3 in \cite{BK.Tak.1979.OpAlg_I} show $1)$, resp.~$2)$ at once. Using $2)$, we directly verify $3)$.
\end{proof}

\begin{prp}\label{PRP.Cstar_Rep}
Let $A$ be a $C^{*}$-algebra. There exists Hilbert space $H$ and faithful $^{*}$-representation $\pi:A\longrightarrow\BII(H)$. If $A$ is unital, then we may ask $\pi$ to be unital.
\end{prp}
\begin{proof}
Apply Theorem I.9.18 in \cite{BK.Tak.1979.OpAlg_I}.
\end{proof}


\pagebreak


Faithful $^{*}$-representations of $C^{*}$-algebras are direct sums of cyclic $^{*}$-representations. The latter arise as GNS-constructions, which are standard constructions associated to positive functionals on $C^{*}$-algebras \lc{}cf.~Theorem I.9.14 in \cite{BK.Tak.1979.OpAlg_I}\rc{}. If we further demand normality given $W^{*}$-algebras, i.e.~$\sigma$-weak closed $C^{*}$-algebras, then we obtain semi-cyclic $^{*}$-representations \lc{}cf.~Definition VII.1.5 in \cite{BK.Tak.2003.OpAlg_II}\rc{}. Relevant to us are canonical left-~and right-actions associated to f.s.n.~traces constructed in Subsection \ref{SSEC.B_SMO_Wstar_Trace}.

\begin{prp}\label{PRP.Cstar_PO}
Let $A$ be a $C^{*}$-algebra.

\begin{itemize}
\item[1)] The partial order generated on $A_{h}$ by the proper cone $A_{+}$ is given by

\begin{align}\label{EQ.PRP.Cstar_PO_1}
x\geq y\Leftrightarrow x-y\in A_{+}
\end{align}

\begin{reapply}
\end{reapply}

\noindent for all $x,y\in A_{h}$. Using algebra involution as anti-linear involution on $A$, the set $A_{+}$ of positive elements generates the partial order.

\item[2)] The partial order generated on $A_{h}^{*}$ by the proper cone $A_{+}^{*}$ is given by

\begin{align}\label{EQ.PRP.Cstar_PO_2}
\mu\geq\eta\Leftrightarrow\mu-\eta\in A_{+}
\end{align}

\begin{reapply}
\end{reapply}

\noindent for all $\mu,\eta\in A_{h}^{*}$. Using pointwise conjugation as anti-linear involution on $A^{*}$, the set $A_{+}$ of positive elements generates the partial order.
\end{itemize}
\end{prp}
\begin{proof}
Apply Theorem I.6.1 and Proposition III.2.1 in \cite{BK.Tak.1979.OpAlg_I}.
\end{proof}

\begin{dfn}\label{DFN.Cstar_PO}
Let $A$ be a $C^{*}$-algebra. We equip $A_{h}$ with the partial order defined by Equation \ref{EQ.PRP.Cstar_PO_1}, resp.~$A_{h}^{*}$ with the partial order defined by Equation \ref{EQ.PRP.Cstar_PO_2}.
\end{dfn}

\begin{rem}
If $H$ is a Hilbert space, then partial order on $\BII(H)_{h}$ given by $1)$ in Definition \ref{DFN.Unbd_PO} is the one fixed here by Definition \ref{DFN.Cstar_PO}. Note Example \ref{BSP.Wstar_CP_II} shows all $^{*}$-homomorphisms are positivity-preserving. Altogether, we know partial orders on $C^{*}$-algebras reduce to Definition \ref{DFN.Unbd_PO} by Proposition \ref{PRP.PO_II} and Proposition \ref{PRP.Cstar_Rep}.
\end{rem}

We use three standard constructions for $C^{*}$-algebras: generation, direct sums and tensor products. Definition \ref{DFN.Cstar_Generated} gives generated $C^{*}$-algebras. Let $A$ be a $C^{*}$-algebra and $S\subset A$. Let $\textrm{Poly}(S)$ be the set of all finite polynomials with elements in $S$ or $S^{*}$. We know $\textrm{Poly}(S)^{*}=\textrm{Poly}(S)\subset A$ by construction. For all $n\in\mathbb{N}$, set

\begin{align}\label{EQ.SSEC.A_Fnd_CWstar_1}
\textrm{Poly}(S)^{n}:=\lset{}x\in A\ \vset\ \exists \{\hspace{0.0125cm} y_{k}\hspace{0.0025cm}\}_{k=1}^{n}\subset\textrm{Poly}(S):\ x=\prod_{k=1}^{n}y_{k}\rset{}.
\end{align}

\noindent Note Equation \ref{EQ.SSEC.A_Fnd_CWstar_1} implies the complex linear span $C_{0}^{*}(S):=\langle\hspace{0.0375cm}\bigcup_{n\in\mathbb{N}}\textrm{Poly}(S)^{n}\rangle_{\mathbb{C}}\subset A$ is in fact a $^{*}$-subalgebra. The $C^{*}$-identity is therefore inherited from $A$.

\begin{dfn}\label{DFN.Cstar_Generated}
Let $A$ be a $C^{*}$-algebra. For all $S\subset A$, we call $C^{*}(S):=\overline{C_{0}^{*}(S)}^{\|.\|_{A}}$ the $C^{*}$-algebra generated by $S$. If $\lset{}S_{k}\rset_{k=1}^{n}\subset\PII(A)$, then set $C^{*}\lc{}S_{1},\ldots,S_{n}\rc{}:=C^{*}\lc\bigcup_{k\in K}S_{k}\rc$.
\end{dfn}


\pagebreak


Definition \ref{DFN.Cstar_DS} gives direct sum $C^{*}$-algebras. Let $m\in\mathbb{N}$. For all $n\in\lset{}1,\ldots,m\rset$, let $A_{n}$ be a $C^{*}$-algebra. Let $\oplus_{n=1}^{m}A_{n}$ be the direct sum of Banach spaces. Thus

\begin{align}\label{EQ.SSEC.A_Fnd_CWstar_2}
\|x\|_{\oplus_{n=1}^{m}A_{n}}=\max_{1\leq n\leq m}\hspace{0.025cm} \|x_{n}\|_{A_{n}}
\end{align}

\noindent for all $x=\lc{}x_{1},\ldots,x_{n}\rc\in\oplus_{n=1}^{m}A_{n}$. Multiplication and adjoining is defined on summands. Equation \ref{EQ.SSEC.A_Fnd_CWstar_2} ensures the $C^{*}$-identity.

\begin{dfn}\label{DFN.Cstar_DS}
Let $m\in\mathbb{N}$. For all $n\in\lset{}1,\ldots,m\rset$, let $A_{n}$ be a $C^{*}$-algebra. We call $\oplus_{n=1}^{m}A_{n}$ the direct sum $C^{*}$-algebra of $\lset{}A_{n}\rset_{n=1}^{m}$.
\end{dfn}

Definition \ref{DFN.Cstar_TP} gives tensor product $C^{*}$-algebras. Note we assume nuclearity of at least one factor \lc{}cf.~Definition XV.1.4 in \cite{BK.Tak.2003.OpAlg_III}\rc{}. This ensures unique cross norms up to $^{*}$-isomorphism. For details on $C^{*}$-tensor products, we refer to Section IV.4 in \cite{BK.Tak.1979.OpAlg_I} and Chapter 11 in \cite{BK.Kad_Rin.1997.OpAlg_II}. The latter discusses infinite tensor products.\par
Let $A$ be a $C^{*}$-algebra and $B$ a nuclear $C^{*}$-algebra. We construct minimal $C^{*}$-tensor product $A\otimes B:=A\otimes_{\min}B$ via norm closure of algebraic tensor product $A\odot B$ under the unique norm satisfying the cross norm identity

\begin{align}\label{EQ.SSEC.A_Fnd_CWstar_3}
\|x\otimes y\|_{A\otimes B}=\|x\|_{A}\| y\|_{B}
\end{align}

\noindent for all $x\in A$ and $y\in B$. Multiplication and adjoining are defined on factors and therefore elementary tensors. Equation \ref{EQ.SSEC.A_Fnd_CWstar_3} ensures the $C^{*}$-identity.

\begin{dfn}\label{DFN.Cstar_TP}
Let $A$ be a $C^{*}$-algebra $A$ and $B$ a nuclear $C^{*}$-algebra. We call $A\otimes B$ the $C^{*}$-tensor product of $A$ and $B$.
\end{dfn}

\begin{rem}\label{REM.Cstar_TP}
We are able to tensor suitable bounded linear maps via the algebraic tensor product, in particular bounded linear functionals and $^{*}$-homomorphisms. Both sets of linear maps are stable under $C^{*}$-tensoring.
\end{rem}


\subsubsection*{$\mathbf{W}^{*}$-algebras}

Upon faithful $^{*}$-representation, closures of $C^{*}$-algebras in $\sigma$-weak operator topology are unital $C^{*}$-algebras. This defines $W^{*}$-algebras concretely but it is their $^{*}$-algebra structures which determine $\sigma$-weak operator topology. Definition \ref{DFN.Wstar} gives an equivalent abstract definition as $C^{*}$-algebras which are Banach duals. Their pre-duals are unique up to isometric isomorphism, including noncommutative $L^{1}$-spaces of tracial $W^{*}$-algebras as per Definition \ref{DFN.NC_Int_Lp}. Upon faithful normal $^{*}$-representation as per Proposition \ref{PRP.Wstar_Equivalence}, the induced $w^{*}$-topology is $\sigma$-weak operator topology.

\begin{rem}\label{REM.Wstar_Normal_I}
Proposition \ref{PRP.Wstar_Normal} for weakly continuous faithful $^{*}$-representations as per Proposition \ref{PRP.Wstar_Equivalence} shows normality of $^{*}$-representations is continuity w.r.t.~$w^{*}$-~and $\sigma$-weak operator topology. Unitality is not necessary.
\end{rem}

\begin{dfn}\label{DFN.Wstar}
Let $M$ be a $C^{*}$-algebra. We say that $M$ is a $W^{*}$-algebra if there exists a Banach space $M_{*}$ s.t.~$M=\lc{}M_{*}\rc^{*}$. In this case, we call $M_{*}$ the pre-dual of $M$.
\end{dfn}

\begin{rem}
If $M$ is a $W^{*}$-algebra, then $M_{*}$ is unique up to isometric isomorphism of Banach spaces \lc{}cf.~Corollary III.3.9 in \cite{BK.Tak.1979.OpAlg_I}\rc{}.
\end{rem}

\begin{bsp}\label{BSP.Wstar}
In Subsection \ref{SSEC.B_SMO_Wstar_Trace}, we cover tracial $W^{*}$-algebras. Their pre-duals are noncommutative $L^{1}$-spaces. Fundamental example is $\BII(H)=S^{1}(H)^{*}$ for a Hilbert space $H$ and $S^{1}(H)$ its trace-class operators. The $\sigma$-weak operator topology is defined as the $w^{*}$-topology on $\BII(H)=S^{1}(H)^{*}$ in this case. This mirrors the commutative case of $X$ a locally compact Hausdorff space and $\mathcal{N}$ a $\sigma$-ideal of null sets of the Borel $\sigma$-algebra $\mathfrak{B}(X)$. We define $W^{*}$-algebra $L^{\infty}\lc{}X,\mathcal{N}\rc$ using $\|.\|_{\infty}$ modulo $\mathcal{N}$, i.e.~essential supremum. If $\mathcal{N}=\mathcal{N}_{\mu}$ for $\mu\in C_{c}(X)^{*}$ as per Riesz–Markov–Kakutani theorem \lc{}cf.~Theorem 6.3.4 in \cite{BK.Ped.1989.Analysis_Now}\rc{}, then we know $L^{\infty}\lc{}X,\mu\rc{}:=L^{\infty}\lc{}X,\mathcal{N}_{\mu}\rc{}=L^{1}\lc{}X,\mu\rc^{*}$ depending only on $\mathcal{N}_{\mu}$.
\end{bsp}

\begin{prp}\label{PRP.Wstar_Equivalence}
Let $M$ be a $C^{*}$-algebra. $M$ is a $W^{*}$-algebra if and only if $M$ is unital and there exists a faithful unital $^{*}$-representation $\pi:M\longrightarrow\BII(H)$ satisfying one of the following:

\begin{itemize}
\item[1)] $\pi(M)=\pi(M)''$,

\item[2)] $\pi(M)$ is ($\sigma$-)strongly closed,

\item[3)] $\pi(M)$ is ($\sigma$-)weakly closed.
\end{itemize}

\noindent If $M$ is a $W^{*}$-algebra, then there exists a faithful unital $^{*}$-representation $\pi:M\longrightarrow\BII(H)$ s.t.~the $w^{*}$-topology on $M=\lc{}M_{*}\rc^{*}$ is the $\sigma$-weak operator topology on $\pi(M)\subset\BII(H)$.
\end{prp}
\begin{proof}
For all Hilbert spaces $H$ and $S\subset\BII(H)$, let $S'\subset\BII(H)$ be the commutant of $S$. Theorem II.3.9 and Theorem III.3.5 in \cite{BK.Tak.1979.OpAlg_I} show all claims, with exception of the weak topologies on $M$ and $\pi(M)$ coinciding. Theorem 7.4.2 in \cite{BK.Kad_Rin.1997.OpAlg_II} shows the latter.
\end{proof}

Proposition \ref{PRP.Wstar_Normal} states $\sigma$-weak continuity is normality as per Definition \ref{DFN.Wstar_Normal} for completely positive maps as per Definition \ref{DFN.Wstar_CP}. Note Example \ref{BSP.Wstar_CP_II} shows all $^{*}$-homomorphisms are completely positive. Thus we see all $\sigma$-weakly $^{*}$-homomorphisms are normal, hence all those faithful unital $^{*}$-representations weakly continuous as per Proposition \ref{PRP.Wstar_Equivalence} are also normal, i.e.~Remark \ref{REM.Wstar_Normal_I}. Altogether, we know normality as per $1)$ in Definition \ref{DFN.Wstar_Hom} is a special case of Definition \ref{DFN.Wstar_Normal}. If we consider any $W^{*}$-subalgebras as per $2)$ in Definition \ref{DFN.Wstar_Hom}, then we do not assume unitality unless stated otherwise. For details on the choice of unit, we refer to Subsection \ref{SSEC.B_JFC_Compression}.

\begin{dfn}\label{DFN.Wstar_Hom}
Let $M$ and $N$ be $W^{*}$-algebras.

\begin{itemize}
\item[1)]  A normal $^{*}$-homomorphism $\phi:M\longrightarrow N$ of $W^{*}$-algebras is $\sigma$-weakly continuous $^{*}$-homomorphism.

\item[2)] If $N\subset M$, then $N$ is $W^{*}$-subalgebra of $M$ if $N\subset M$ is normal $^{*}$-homomorphism. If it is also unital, then $N$ is a unital $W^{*}$-subalgebra of $M$.
\end{itemize}
\end{dfn}

Standard constructions for $C^{*}$-algebras specialise to $W^{*}$-algebras. The direct sum construction is unchanged. Definition \ref{DFN.Wstar_Generated} gives generated $W^{*}$-algebras by $\sigma$-weak closure of generated $C^{*}$-algebras. Definition \ref{DFN.Wstar_TP} gives tensor product $W^{*}$-algebras.\par


\pagebreak


\begin{dfn}\label{DFN.Wstar_Generated}
Let $M$ be a $W^{*}$-algebra. For all $S\subset M$, we call $W^{*}(S):=\overline{C_{0}^{*}(S)}^{\w}$ the $W^{*}$-algebra generated by $S$. If $\lset{}S_{k}\rset_{k=1}^{n}\subset\PII(M)$, then set $W^{*}\lc{}S_{1},\ldots,S_{n}\rc{}:=W^{*}\lc\bigcup_{k\in K}S_{k}\rc$.
\end{dfn}

\begin{prp}\label{PRP.Wstar_Generated}
For all unital $C^{*}$-algebras $A$, we have $A=C^{*}\lc\UII(A)\rc$ for the set $\UII(A):=\lset{}x\in A\ \vset\ x^{*}=x^{-1}\rset$ of unitary operators in $A$. For all $W^{*}$-algebras $M$, we have $M=W^{*}\lc{}P(M)\rc$ for the set $P(M):=\lset{}p\in M_{h}\ \vset\ p^{2}=p\rset$ of projections in $M$.
\end{prp}
\begin{proof}
Let $A$ be a unital $C^{*}$-algebra. For all $x\in A_{h}$, $C\lc\specA x\rc{}=C^{*}\lc\UII\lc{}C\lc\specA x\rc\rc\rc$ by Stone-Weierstrass \cite{BK.Ped.1989.Analysis_Now}. Since $\UII\lc{}C\lc\specA x\rc\rc\subset\UII(A)$ in each case, decomposing into real and imaginary parts shows $A=C^{*}\lc\UII(A)\rc$. Let $M$ be a $W^{*}$-algebra. We readily see $M=C^{*}\lc\UII(M)\rc$ by unitality. Theorem 5.2.5 in \cite{BK.Kad_Rin.1997.OpAlg_I} implies $\UII(M)\subset C^{*}\lc{}P(M)\rc$. Thus we combine both to $M\subset C^{*}\lc{}P(M)\rc$, hence $M=W^{*}\lc{}P(M)\rc$ as claimed.
\end{proof}

For our discussion, it commonly suffices to have bounded linear maps preserving strong or weak convergence of uniformly bounded nets. Proposition \ref{PRP.Wstar_Normal} shows such bounded convergence is equivalent to normality if we assume complete positivity.

\begin{prp}\label{PRP.Wstar_BdCon}
Let $M$ be a $W^{*}$-algebra, $S\subset M$ a $^{*}$-subalgebra and $\overline{S}$ its strong closure. For all $x\in\overline{S}$, there exists net $\{x_{k}\}_{k\in K}\subset S$ s.t.~

\begin{align}\label{EQ.PRP.Wstar_BdCon_1}
x=\s\textrm{-}\lim_{k\in K}\hspace{0.025cm} x_{k},\ \sup_{k\in K}\hspace{0.025cm} \|x_{k}\|_{M}\leq \|x\|_{M}.
\end{align}
\end{prp}
\begin{proof}
If $x=0$, then $x\in S$. If $x\neq 0$, then the Kaplansky density theorem yields a net as claimed up to rescaling by a positive constant \lc{}cf.~Theorem 5.3.5 in \cite{BK.Kad_Rin.1997.OpAlg_I}\rc{}.
\end{proof}

\begin{dfn}\label{DFN.Wstar_BdCon_I}
Let $M$ be a $W^{*}$-algebra. We call a net $\{x_{k}\}_{k\in K}\subset M$ bounded strongly convergent if it is bounded and converges strongly, resp.~bounded weakly convergent if it is bounded and converges weakly.
\end{dfn}

\begin{ntn}
Let $x=\bds$-$\lim_{k\in K}x_{k}$ denote bounded strong and $x=\bdw$-$\lim_{k\in K}x_{k}$ bounded weak convergence of nets.
\end{ntn}

\begin{rem}\label{REM.Wstar_BdCon_I}
The uniform boundedness principle shows bounded strong and strong convergence coincide on sequences \lc{}cf.~Theorem 2.2.9 in \cite{BK.Ped.1989.Analysis_Now}\rc{}. Equally, bounded weak and weak convergence coincide on sequences.
\end{rem}

\begin{dfn}\label{DFN.Wstar_BdCon_II}
Let $\phi:M\longrightarrow N$ be a bounded linear map of $W^{*}$-algebras.

\begin{itemize}
\item[1)] We call $\phi$ bounded strongly continuous if for all nets $\{x_{k}\}_{k\in K}\subset M$, $x=\bds$-$\lim_{k\in K}x_{k}$ implies $\phi(x)=\bds$-$\lim_{k\in K}\phi(x_{k})$,

\item[2)] We call $\phi$ bounded weakly continuous if for all nets $\{x_{k}\}_{k\in K}\subset M$, $x=\bdw$-$\lim_{k\in K}x_{k}$ implies $\phi(x)=\bdw$-$\lim_{k\in K}\phi(x_{k})$.
\end{itemize} 
\end{dfn}

\begin{rem}\label{REM.Wstar_BdCon_II}
Definition \ref{DFN.Wstar_BdCon_II} extends to bounded multi-linear maps. We commonly use multiplication in $W^{*}$-algebras is bounded strongly continuous and therefore further sequentially strongly continuous by Remark \ref{REM.Wstar_BdCon_I}.
\end{rem}


\subsubsection*{Normal, completely positive and completely Markovian maps}

We consider properties of bounded linear maps of $C^{*}$-~and $W^{*}$-algebras. Completely positive normal bounded linear maps of $W^{*}$-algebras are continuous in all operator topologies we use and stable under tensoring. Examples are positive bounded normal functionals on and normal $^{*}$-homomorphisms of $W^{*}$-algebras, as well as compression maps. The notions of completely positive map and completely Markovian map are used to define completely Markovian semigroups \cite{ART.Dav.1979.Quantum_Markov_SG_Generators}\cite{ART.Dav_Lin.1992.Noncommutative_Markov_SG_I}\cite{ART.Dav_Lin.1993.Noncommutative_Markov_SG_II} describing irreversible time-evolution of dissipative quantum systems weakly coupled to a heat bath \cite{BK.Bra.1987.OpAlg_Quantum_StM_I}\cite{BK.Bra.1987.OpAlg_Quantum_StM_II}\cite{BK.Dav.1976.Quantum_Markov_SG}\cite{BK.Gar_Zol.2004.Quantum_Noise}\cite{BK.Ohy_Pet.1993.Rel_Ent}\cite{BK.Ste_vLee.2013.Full_Quantum_StM}. For details on the latter, we refer to Subsection \ref{SSEC.QOT_AC_HSG}.\par
Normality is preservation of suprema under a given map. Unique suprema exist for $W^{*}$-algebras. For all bounded increasing nets $\{x_{k}\}_{k\in K}\subset M_{h}$ in a given $W^{*}$-algebra $M$, we have unique supremum $\sup_{k\in K}x_{k}\in M_{h}$ in partial order. We furthermore have $\sup_{k\in K}x_{k}=\s$-$\lim_{k\in K}x_{k}$ in $M$. Lemma 5.1.4 in \cite{BK.Kad_Rin.1997.OpAlg_I} shows both statements.

\begin{dfn}\label{DFN.Wstar_Normal}
Let $\phi:M\longrightarrow N$ be a positivity-preserving bounded linear map of $W^{*}$-algebras. We call $\phi$ normal if for all bounded increasing nets $\{x_{k}\}_{k\in K}\subset M_{h}$, get

\begin{align}\label{EQ.DFN.Wstar_Normal_1}
\phi\vstretch{1.05}{\Bigg(}\hspace{0.0075cm} \sup_{k\in K}\hspace{0.025cm} x_{k}\vstretch{1.05}{\Bigg)}=\sup_{k\in K}\hspace{0.025cm} \phi(x_{k}).
\end{align}
\end{dfn}

In the general noncommutative setting, positivity-preservation is not stable under tensoring. The latter requires complete positivity. For all completely positive maps of $W^{*}$-algebras, Proposition \ref{PRP.Wstar_Normal} shows normality is equivalent to $\sigma$-weak continuity. Example \ref{BSP.Wstar_CP_II} therefore leads to Definition \ref{DFN.Wstar_Hom}. Full matrix algebras are nuclear $C^{*}$-algebras. For all $n\in\mathbb{N}$, let $I_{n}\in M_{n}(\mathbb{C})$ be the identity. For all bounded linear maps $\phi:A\longrightarrow B$ of $C^{*}$-algebras, the bounded linear maps $\phi\otimes\id_{M_{n}(\mathbb{C})}:A\otimes M_{n}(\mathbb{C})\longrightarrow B\otimes M_{n}(\mathbb{C})$ of $C^{*}$-algebras obtained for all $n\in\mathbb{N}$ are determined on algebraic tensor products.

\begin{dfn}\label{DFN.Wstar_CP}
We call a bounded linear map $\phi:A\longrightarrow B$ of $C^{*}$-algebras completely positive if $\phi\otimes\id_{M_{n}(\mathbb{C})}:A\otimes M_{n}(\mathbb{C})\longrightarrow B\otimes M_{n}(\mathbb{C})$ is positivity-preserving for all $n\in\mathbb{N}$.
\end{dfn}

\begin{bsp}\label{BSP.Wstar_CP_I}
All positivity-preserving bounded linear functionals $\mu:A\longrightarrow\mathbb{C}$ of $C^{*}$-algebras are completely positive \lc{}cf.~Corollary IV.3.5 in \cite{BK.Tak.1979.OpAlg_I}\rc{}.
\end{bsp}

\begin{bsp}\label{BSP.Wstar_CP_II}
If $\phi:A\longrightarrow B$ is a $^{*}$-homomorphisms of $C^{*}$-algebras, then $\phi(x^{*}x)=\phi(x)^{*}\phi(x)$ for all $x\in A$ ensures $\phi$ is positivity-preserving. Since each $\phi\otimes\id_{M_{n}(\mathbb{C})}$ itself is a $^{*}$-homomorphism if $\phi$ is, $^{*}$-homomorphisms are completely positive. Proposition \ref{PRP.Wstar_Normal} shows $\sigma$-weak continuity is normality, i.e.~Equation \ref{EQ.DFN.Wstar_Normal_1}, for $^{*}$-homomorphisms.
\end{bsp}

\begin{bsp}\label{BSP.Wstar_CP_III}
Let $M$ be a $W^{*}$-algebra and $p\in M$ a projection. We obtain $W^{*}$-algebra $M[p]:=pMp$ and define positivity-preserving compression map $\comp:M\longrightarrow M[p]$ by setting $\comp x:=pxp$ for all $x\in M$. For all $n\in\mathbb{N}$, $p\otimes I_{n}\in M\otimes M_{n}(\mathbb{C})$ is a projection and $\comp\otimes\id_{M_{n}(\mathbb{C})}=\com_{p\otimes I_{n}}$ upon repeat construction. Thus $\comp$ is completely positive. We define compression maps in Definition \ref{DFN.Compression_Abstract_Bd}. Note $2)$ in Proposition \ref{PRP.Wstar_Trace_NCE_II} shows $\comp$ is the unique noncommutative conditional expectation from $M$ to $M[p]$.
\end{bsp}

\begin{prp}\label{PRP.Wstar_Normal}
For all completely positive maps $\phi:M\longrightarrow N$ of $W^{*}$-algebras, the following are equivalent:

\begin{itemize}
\item[1)] $\phi$ is normal,

\item[2)] $\phi$ is $\sigma$-weakly continuous,

\item[3)] $\phi$ is $\sigma$-strongly continuous,

\item[4)] $\phi$ is bounded weakly continuous,

\item[5)] $\phi$ is bounded strongly continuous.
\end{itemize}
\end{prp}
\begin{proof}
Proposition III.2.2.2 in \cite{BK.Bla.2006.OpAlg} shows $1)$ to $3)$. As $\phi$ is bounded, the unit ball in $M$ is mapped to a bounded ball in $N$. Note $\sigma$-strong and strong, as well as $\sigma$-weak and weak topologies are equivalent on norm bounded sets of $W^{*}$-algebras \lc{}cf.~Lemma II.2.5 in \cite{BK.Tak.1979.OpAlg_I}\rc{}. For all bounded increasing nets $\{x_{k}\}_{k\in K}\subset M_{h}$, $\sup_{k\in K}x_{k}$ is the $\sigma$-strong and therefore $\sigma$-weak limit of $\{x_{k}\}_{k\in K}$. Equivalence of $1)$ and $4)$, as well as $1)$ and $5)$, thus hold by equivalence of the operator topologies on norm bounded sets.
\end{proof}

\begin{rem}\label{REM.Wstar_Normal_II}
By Remark \ref{REM.Wstar_BdCon_I} and Proposition \ref{PRP.Wstar_Normal}, completely positive normal bounded linear maps of $W^{*}$-algebras are sequentially strongly and sequentially weakly continuous. We use this throughout our discussion.
\end{rem}

We refer to Section IV.4 in \cite{BK.Tak.1979.OpAlg_I} for details on $W^{*}$-tensor products. We do not assume nuclearity for tensor products of $W^{*}$-algebras as their construction uses unique minimal $C^{*}$-tensor products. Let $M$ and $N$ be $W^{*}$-algebras. Their minimal $C^{*}$-tensor product is $M\otimes_{\min} N$. Let $M_{*}$ and $N_{*}$ denote their respective pre-dual. Get $M_{*}\odot N_{*}\subset \lc{}M\otimes_{\min} N\rc^{*}$ for the algebraic tensor product of pre-duals by letting 

\begin{align}\label{EQ.SSEC.A_Fnd_CWstar_4}
\lc\mu\otimes\eta\rc\lc{}x\otimes y\rc{}:=\mu(x)\eta(y)
\end{align}

\noindent for all $\mu\otimes\eta\in M_{*}\otimes N_{*}$ and $x\otimes y\in M\otimes_{\min} N$.

\begin{dfn}\label{DFN.Wstar_TP}
Let $M$ and $N$ be $W^{*}$-algebras. Set 

\begin{align}\label{EQ.DFN.Wstar_TP_1}
M_{*}\otimes N_{*}:=\overline{M_{*}\odot N_{*}}\subset\lc{}M\otimes_{\min} N\rc^{*}
\end{align}

\noindent using norm closure. We call $M\otimes N:=\lc{}M_{*}\otimes N_{*}\rc^{*}$ the $W^{*}$-tensor product of $M$ and $N$.
\end{dfn}

\begin{lem}\label{LEM.Wstar_TP}
Let $\phi:M_{0}\longrightarrow M_{1}$ and $\bpsi:N_{0}\longrightarrow N_{1}$ be completely positive normal maps of $W^{*}$-algebras. We define completely positive normal map $\phi\otimes\bpsi:M_{0}\otimes N_{0}\longrightarrow M_{1}\otimes N_{1}$ by setting $\lc\phi\otimes\bpsi\rc\lc{}x\otimes y\rc{}:=\phi(x)\otimes\bpsi(y)$ for all $x\in M_{0}$ and $y\in N_{0}$.
\end{lem}
\begin{proof}
By Proposition \ref{PRP.Wstar_Normal}, this is Proposition IV.5.13 in \cite{BK.Tak.1979.OpAlg_I}.
\end{proof}

\begin{cor}\label{COR.Wstar_TP}
Let $\phi:M_{0}\longrightarrow M_{1}$ and $\bpsi:N_{0}\longrightarrow N_{1}$ be normal $^{*}$-homomorphisms of $W^{*}$-algebras. We define normal $^{*}$-homomorphism $\phi\otimes\bpsi:M_{0}\otimes N_{0}\longrightarrow M_{1}\otimes N_{1}$ by setting $\lc\phi\otimes\bpsi\rc\lc{}x\otimes y\rc{}:=\phi(x)\otimes\bpsi(y)$ for all $x\in M_{0}$ and $y\in N_{0}$. If $\phi$ and $\bpsi$ are unital, then $\phi\otimes\bpsi$ is.
\end{cor}
\begin{proof}
Example \ref{BSP.Wstar_CP_II} and Lemma \ref{LEM.Wstar_TP} yield completely positive normal $\phi\otimes\bpsi$. By Proposition \ref{PRP.PO_II} and Proposition \ref{PRP.Wstar_Normal}, $\phi\otimes\bpsi$ intertwines adjoining and is $\sigma$-strongly continuous. This is equivalent to bounded strong convergence for bounded nets. Let $S$ be the linear span of all elementary tensors. By strong density, Proposition \ref{PRP.Wstar_BdCon} shows $M_{0}\otimes N_{0}$ is the bounded strong closure of $S$. As multiplication is bounded strongly continuous and $\phi\otimes\bpsi$ is bounded, we directly verify our claim on elementary tensors.
\end{proof}

\begin{dfn}\label{DFN.Wstar_CP_Markovian}
We call a completely positive map $\phi:A\longrightarrow A$ of unital $C^{*}$-algebras Markovian if $\phi(x)\leq \|x\|_{A}1_{A}$ for all $x\in A_{+}$. We call such maps completely Markovian if $\phi\otimes\id_{M_{n}(\mathbb{C})}:A\otimes M_{n}(\mathbb{C})\longrightarrow A\otimes M_{n}(\mathbb{C})$ is Markovian for all $n\in\mathbb{N}$.
\end{dfn}


\subsection{Functional calculus}\label{SSEC.A_Fnd_FC}

Standard references for continuous and bounded measurable functional calculus for\linebreak $C^{*}$-, resp.~$W^{*}$-algebras are  \cite{BK.Kad_Rin.1997.OpAlg_I} and \cite{BK.Tak.1979.OpAlg_I}. Standard references for spectral integration and functional calculus of self-adjoint unbounded operators are \cite{BK.Ped.1989.Analysis_Now} and \cite{BK.Sch.2012.Unbounded_Operators}.


\subsubsection*{Integration of spectral measures}

Let $H$ be a Hilbert space. Spectral measures of self-adjoint unbounded operators on $H$ are projection-valued measures taking values in $\BII(H)$. Image lattices of projections are noncommutative Borel $\sigma$-algebras.

\begin{ntn}
For all $n\in\mathbb{N}$ and Borel measurable $X\subset\mathbb{C}^{n}$, let $\mathfrak{B}(X)$ denote the Borel $\sigma$-algebra of $X$. Let $\chi_{Z}$ denote the characteristic function of a set $Z\subset\mathbb{C}^{n}$.
\end{ntn}

\begin{dfn}\label{DFN.SpecInt_Measure}
Let $X\in\mathfrak{B}\lc\mathbb{C}^{n}\rc$. A map $E:\mathfrak{B}(X)\longrightarrow\BII(H)$ is a spectral measure on $X$ with values in $\BII(H)$ if

\begin{itemize}
\item[1)] $E(X)=I$ and $E(Z)$ is a projection for all $Z\in\mathfrak{B}(X)$,

\item[2)] $Z\mapsto E^{u}(Z):=\lgl E(Z)(u),u\rgl_{H}$ is a measure on $X$ for all $u\in H$.
\end{itemize}

\noindent Let $E$ be a spectral measure on $X$. Its support $\supp E$ is the set of all $x\in X$ s.t.~$E\lc{}N_{x}\rc\neq 0$ for all open neighbourhoods $N_{x}$ of $x$. Its null ideal is $\mathcal{N}\lc{}E\rc{}:=\lset{}Z\in\mathfrak{B}(X)\ \vset\ E(Z)=0\rset$.
\end{dfn}

Spectral measures $E:\mathfrak{B}(\mathbb{R})\longrightarrow\BII(H)$ map bijectively to resolutions of the identity $\lset{}E\lc\lc{}-\infty,\lambda\rb\rc\rset_{\lambda\in\mathbb{R}}$ \lc{}cf.~Theorem 4.6 in \cite{BK.Sch.2012.Unbounded_Operators}\rc{}. A spectral measure $E:\mathfrak{B}(\mathbb{R})\longrightarrow\BII(H)$ is thus determined by its resolution of the identity $\lset{}E\lc\lc{}-\infty,\lambda\rb\rc\rset_{\lambda\in\mathbb{R}}$.

\begin{prp}
For all spectral measures $E$ on $X\in\mathfrak{B}\lc\mathbb{C}^{n}\rc$, we have

\begin{itemize}
\item[1)] $E^{u}$ is a finite measure for all $u\in H$,

\item[2)] $\supp E$ is minimal among closed $Z\in\mathfrak{B}(X)$ s.t.~$E(Z)=I_{H}$.
\end{itemize}
\end{prp}
\begin{proof}
By definition of spectral measures.
\end{proof}


\pagebreak


The null ideal $\mathcal{N}\lc{}E\rc$ of a given spectral measure $E$ yields notions of $E$-a.e.~defined and $E$-a.e.~finite map. The set of all $E$-a.e.~finite maps $g:X\longrightarrow\mathbb{C}$ is the domain of spectral integration w.r.t.~$E$.

\begin{dfn}
Let $E$ be spectral measure on $X\in\mathfrak{B}\lc\mathbb{C}^{n}\rc$. Let $\SII\lc{}E\rc$ denote the set of all $E$-a.e.~defined Borel measurable $g:X\longrightarrow\mathbb{C}$ s.t.~$\absv{1}{g}$ is $E$-a.e.~finite. We say that $\lset{}Z_{k}\rset_{k\in\mathbb{N}}\subset\mathfrak{B}(X)$ is a bounding sequence for $\mathcal{G}\subset\SII\lc{}E\rc$ if

\begin{itemize}
\item[1)] $Z_{k}\subset Z_{k+1}$ for all $k\in\mathbb{N}$ and $E\lc\bigcup_{k\in\mathbb{N}}Z_{k}\rc{}=I$,

\item[2)] $\absv{1.15}{g\vert_{Z_{k}}}$ is bounded for all $k\in\mathbb{N}$.
\end{itemize}
\end{dfn}

\begin{rem}
For all spectral measures $E$ on $X\in\mathfrak{B}\lc\mathbb{C}^{n}\rc$ and finite $\mathcal{G}\subset\SII\lc{}E\rc$, there exists a bounding sequence \lc{}cf.~Subsection 4.3.2 in \cite{BK.Sch.2012.Unbounded_Operators}\rc{}.
\end{rem}

Let $E$ be a spectral measure on $X\in\mathfrak{B}\lc\mathbb{C}^{n}\rc$. We define spectral integration as per Lemma 4.11 and Theorem 4.13 in \cite{BK.Sch.2012.Unbounded_Operators}. Theorem 4.16 and Subsection 4.3.3 in \cite{BK.Sch.2012.Unbounded_Operators} show fundamental properties. For all simple functions $g=\sum_{l=1}^{n}c_{l}\chi_{Z_{l}}$ on $X$, the spectral integral of $g$ w.r.t.~$E$ is defined by

\begin{align}\label{EQ.SSEC.A_Fnd_FC_1}
\mathrm{I}_{E}(g):=\sum_{l=1}^{n}c_{l}E\lc{}Z_{l}\rc{}.
\end{align}

\noindent Lemma 4.11 in \cite{BK.Sch.2012.Unbounded_Operators} states $\dblv{}\mathrm{I}_{E}(g)\dblv_{\BII(H)}\leq\sup_{x\in X}\absv{1.15}{g(x)}$ in each case. Density of simple functions in uniform norm extends spectral integration w.r.t.~$E$ to all bounded Borel functions on $X$. For all bounded Borel functions $g:X\longrightarrow\mathbb{C}$ and simple functions $\{g_{n}\}_{n\in\mathbb{N}}$ on $X$ $\|.\|_{\infty}$-converging to $g$, get $\mathrm{I}_{E}(g)=\|.\|_{\BII(H)}$-$\lim_{n\in\mathbb{N}}\mathrm{I}_{E}\lc{}g_{n}\rc$.\par
Let $g\in\SII\lc{}E\rc$. The domain of $\mathrm{I}_{E}$ is defined by

\begin{align}\label{EQ.SSEC.A_Fnd_FC_2}
\dom \mathrm{I}_{E}(g):=\lset{}u\in H\ \vset\ \int_{X}\absv{1.15}{g(x)}^{2}dE^{u}<\infty\rset{}.
\end{align} 

\noindent For all $u\in H$, we have $u\in\dom\mathrm{I}_{E}(g)$ if and only if

\begin{align}\label{EQ.SSEC.A_Fnd_FC_3}
\mathrm{I}_{E}(g)(u):=\|.\|_{H}\textrm{-}\lim_{k\in\mathbb{N}}\hspace{0.025cm} \mathrm{I}_{E}\lc{}g\chi_{Z_{k}}\rc{}(u)
\end{align}

\noindent exists for a bounding sequence $\lset{}Z_{k}\rset_{k\in\mathbb{N}}$ of $g$. If $u\in\dom\mathrm{I}_{E}(g)$, then $\mathrm{I}_{E}(g)(u)$ exists and is independent of choice of bounding sequence of $g$ by Theorem 4.13 in \cite{BK.Sch.2012.Unbounded_Operators}. 

\begin{dfn}
Let $E$ be a spectral measure on $X\in\mathfrak{B}\lc\mathbb{C}^{n}\rc$. For all $g\in\SII\lc{}E\rc$, we call $\int gdE:=\mathrm{I}_{E}(g)$ the spectral integral of $g$ w.r.t.~$E$.
\end{dfn}

\begin{rem}\label{REM.SpecInt_Identities}
Let $g\in\SII\lc{}E\rc$. Its domain as per Equation \ref{EQ.SSEC.A_Fnd_FC_2} and the identity

\begin{align}\label{EQ.REM.SpecInt_Identities_1}
\lgl\mathrm{I}_{E}(g)(u),u\rgl_{H}=\int_{X}g(x)dE^{u}
\end{align}

\noindent for all $u\in\dom\mathrm{I}_{E}(g)$, i.e.~$\dblv{}\mathrm{I}_{E}(g)(u)\dblv_{H}^{2}=\int_{X}\absv{1.15}{g(x)}^{2}dE^{u}<\infty$, determine $\mathrm{I}_{E}(g)$.
\end{rem}


\pagebreak


\begin{prp}\label{PRP.SpecInt_Operator}
Let $E$ be a spectral measure on $X\in\mathfrak{B}\lc\mathbb{C}^{n}\rc$. For all $g\in\SII\lc{}E\rc$, $\mathrm{I}_{E}(g)$ is a closed normal operator s.t.~$\mathrm{I}_{E}(g)^{*}=\mathrm{I}_{E}\lc\overline{g}\rc$ and $\mathrm{I}_{E}(g)^{*}\mathrm{I}_{E}(g)=\mathrm{I}_{E}(g)\mathrm{I}_{E}(g)^{*}=\mathrm{I}_{E}\lc{}g\overline{g}\rc$.
\end{prp}
\begin{proof}
Apply Theorem 4.16 in \cite{BK.Sch.2012.Unbounded_Operators}.
\end{proof}

\begin{rem}\label{REM.SpecInt_Operator}
Proposition \ref{PRP.SpecInt_Operator} defines invertible map $E\mapsto\mathrm{I}_{E}\lc\id_{\mathbb{R}}\rc$ from all spectral measures $E:\mathfrak{B}(\mathbb{R})\longrightarrow\BII(H)$ to $\UBII(H)_{h}$. Note invertibility is the spectral theorem for self-adjoint unbounded operators \lc{}cf.~Theorem 5.7 in \cite{BK.Sch.2012.Unbounded_Operators}\rc{}.  
\end{rem}


\subsubsection*{Bounded measurable functional calculus}

Functional calculus of self-adjoint unbounded operators is based on the use of spectral measures. We construct these using bounded measurable functional calculus for $W^{*}$-algebras. The latter in turn extends continuous functional calculus for unital $C^{*}$-algebras.\par
The choice of unit matters. If a Banach $^{*}$-algebra is unital, then the unit is unique. If however $N\subset M$ is a $W^{*}$-subalgebra s.t.~$1_{N}\neq 1_{M}$, then all normal elements in $N$ have two a priori distinct bounded measurable functional calculi. Equation \ref{EQ.LEM.Wstar_FC_1} shows how they may differ. If they differ, then they differ only at zero but generate distinct spectral measures. This impacts spectral integration, in particular taking inverses.

\begin{dfn}\label{DFN.Cstar_Unitalisation}
Let $B$ be a unital $C^{*}$-algebra. For all $C^{*}$-subalgebras $A\subset B$, we call $A[1_{B}]=C^{*}\lc{}A,1_{B}\rc$ the unitalisation of $A$ in $B$.
\end{dfn}

\begin{prp}\label{PRP.Cstar_Unitalisation}
Let $A$ and $B$ be unital $C^{*}$-algebras and $A\subset B$ a $C^{*}$-subalgebra. If $A\subset B$ is not a unital $C^{*}$-subalgebra, then $A[1_{B}]=A\oplus\langle 1_{B}-1_{A}\rangle_{\mathbb{C}}$.
\end{prp}
\begin{proof}
Get $1_{B}-1_{A}\in A[1_{B}]$ and $\lc{}1_{B}-1_{A}\rc{}A=A\lc{}1_{B}-1_{A}\rc{}=0$. Thus $A\oplus\langle 1_{B}-1_{A}\rangle_{\mathbb{C}}$.
\end{proof}

\begin{dfn}\label{DFN.Cstar_FC}
Let $A$ be a $C^{*}$-algebra.

\begin{itemize}
\item[1)] We call $x\in A$ normal if $x^{*}x=xx^{*}$.

\item[2)] Let $A$ be unital. Set $\GL(A):=\lset{}x\in A\ \vset\ x^{-1}\in A\rset$. For all normal $x\in A$, its spectrum in $A$ is $\specA x:=\lset\lambda\in\mathbb{C}\ \vset\ x-\lambda 1_{A}\notin\GL(A)\rset$.
\end{itemize}
\end{dfn}

Lemma \ref{LEM.Cstar_FC} states continuous functional calculus for unital $C^{*}$-algebras. For all normal $x\in A$ in a unital $C^{*}$-algebra $A$, Example \ref{BSP.Cstar_Commutative} explains how $C\lc\specA x\rc$ is a $C^{*}$-algebra using uniform norm.

\begin{lem}\label{LEM.Cstar_FC}
Let $A$ be a unital $C^{*}$-algebra. If $x\in A$ is normal, then

\begin{itemize}
\item[1)] $\specA x\subset\mathbb{C}$ is non-empty and compact,

\item[2)] there exists unital $^{*}$-isomorphism $\Gamma_{x,A}:C\lc\specA x\rc\longrightarrow C^{*}\lc{}x,x^{*},1_{A}\rc$,

\item[3)] $\Gamma_{x,A}$ is determined by unitality and $\Gamma_{x,A}\lc\id_{\specA x}\rc{}=x$.
\end{itemize}
\end{lem}
\begin{proof}
Get $1)$ by Proposition I.4.2, resp.~$2)$ and $3)$ by Proposition I.4.6 in \cite{BK.Tak.1979.OpAlg_I}.
\end{proof}

\begin{rem}\label{REM.Cstar_FC}
Let $x\in A$ be normal. We call $\Gamma_{x,A}$ the continuous functional calculus of $x$ in $A$. For all $x\in C\lc\specA x\rc$, set $g(x):=\Gamma_{x,A}(g)$. We adopt analogues convention for all functional calculus. If $\specA x\subset X\subset\mathbb{C}$ for locally compact Hausdorff $X$, then $g(x)=h(x)$ for all $g,h\in C_{0}(X)$ s.t.~$g\vert_{\specA x}=h\vert_{\specA x}$.
\end{rem}

Corollary \ref{COR.Cstar_FC_I} shows continuous functional calculus extends uniquely to normal elements in non-unital $C^{*}$-algebras if we restrict to functions vanishing at zero. Note it further shows choice of unit only involves values at zero.

\begin{cor}\label{COR.Cstar_FC_I}
Let $B$ be a unital $C^{*}$-algebra and $A\subset B$ a $C^{*}$-subalgebra. If $x\in A$ is normal and $\pi:A\longrightarrow\BII(H)$ a faithful $^{*}$-representation, then

\begin{itemize}
\item[1)] $\specB x\setminus\lset{}0\rset=\specBH\pi(x)\setminus\lset{}0\rset$,

\item[2)] $\Gamma_{x,B}(g)\in A$ and $\pi\lc\Gamma_{x,B}(g)\rc{}=\Gamma_{\pi(x),\BII(H)}(g)$ for all $g\in C_{0}\lc\specBH\pi(x)\setminus\lset{}0\rset\rc$.
\end{itemize}
\end{cor}
\begin{proof}
Get $1)$ by Proposition \ref{PRP.Cstar_Unitalisation}. Instead of $\specBH\pi(x)\setminus\lset{}0\rset$, we consider compact $K\subset\mathbb{C}$ s.t.~$\lset{}0\rset\cup\specB x\cup\specBH\pi(x)\subset K$ as per Remark \ref{REM.Cstar_FC}. If $g$ is a polynomial on $K$ vanishing at zero, then it is expressed without the constant function. Thus $\Gamma_{x,B}(g)\in A$ and $\pi\lc\Gamma_{x,B}(g)\rc{}=\Gamma_{\pi(x),\BII(H)}(g)$ by Lemma \ref{LEM.Cstar_FC}. If $g\in C(K)$ vanishes at zero, then we approximate $g$ uniformly in norm by polynomial on $K$ vanishing at zero. We conclude by boundedness of $^{*}$-homomorphisms.
\end{proof}

\begin{cor}\label{COR.Cstar_FC_II}
Let $A$ be a $C^{*}$-algebra. For all $x\in A$, we have

\begin{itemize}
\item[1)] $x\in A_{h}$ if and only if $\specA x\subset\mathbb{R}$,

\item[2)] $x\in A_{+}$ if and only if $\specA x\subset [0,\infty)$,

\item[3)] $x=x_{+}-x_{-}$ for $x_{+}:=\max\{x,0\},x_{-}:=-\min\{x,0\}\in A_{+}$ if $x\in A_{h}$. 
\end{itemize}
\end{cor}
\begin{proof}
By Corollary \ref{COR.Cstar_FC_I}, we assume $A$ is unital without loss of generality. Thus $1)$ and $2)$ are Proposition I.4.3 and Theorem I.6.1 in \cite{BK.Tak.1979.OpAlg_I}. Writing $g(x):=\Gamma_{x,A}(g)$ in each case, we see $3)$ is decomposition in Proposition \ref{PRP.Cstar_PO} to have proper cone.
\end{proof}

Lemma \ref{LEM.Wstar_FC} extends to bounded measurable functional calculus. Corollary \ref{COR.FC_Preservation} shows bounded measurable calculus of self-adjoint elements is preserved under normal unital $^{*}$-homomorphisms. In the proof of Lemma \ref{LEM.Wstar_FC}, abstract spectral measures yield bounded measurable functional calculus. Note functional calculus of self-adjoint unbounded operators instead uses concrete ones as it assumes faithful normal unital $^{*}$-representations as per Remark \ref{REM.FC_Bd} in general. In Subsection \ref{SSEC.B_SMO_CLRA}, we unify these approaches for spaces of measurable operators.

\begin{prp}\label{PRP.Wstar_Unitalisation}
If $N\subset M$ is a unital $W^{*}$-subalgebra, then $N[1_{M}]=N$. If $N\subset M$ is a non-unital $W^{*}$-subalgebra, then $N[1_{M}]=N\oplus\langle 1_{M}-1_{N}\rangle_{\mathbb{C}}$.
\end{prp}
\begin{proof}
Proposition \ref{PRP.Cstar_Unitalisation}. Note $C^{*}$-direct sums of $W^{*}$-algebras are $W^{*}$-algebras.
\end{proof}

\begin{lem}\label{LEM.Wstar_FC}
Let $M$ be a $W^{*}$-algebra. If $x\in M$ is normal, then there exists unique $\sigma$-ideal $\mathcal{N}_{x,M}$ of null sets of the Borel $\sigma$-algebra $\mathfrak{B}\lc\specM x\rc$ s.t.~

\begin{itemize}
\item[1)] $\lc{}L^{\infty}\lc\specM x,\mathcal{N}_{x,M}\rc{},\|.\|_{\infty}\rc$ is a $W^{*}$-algebra s.t.~$C\lc\specM x\rc$ is $\sigma$-weakly dense,

\item[2)] $\Gamma_{x,M}$ extends to a normal unital $^{*}$-isomorphism

\begin{align}\label{EQ.LEM.Wstar_FC_1}
\Gamma_{x,M}:L^{\infty}\lc\specM x,\mathcal{N}_{x,M}\rc\longrightarrow W_{M}^{*}(x):=W^{*}(x,x^{*},1_{M}),
\end{align}

\begin{reapply}
\end{reapply}

\item[3)] $\Gamma_{x,M}$ is determined by unitality and $\Gamma_{x,M}\lc\id_{\specM x}\rc{}=x$.
\end{itemize}
\end{lem}
\begin{proof}
Let $x\in M_{h}$. For details and the normal case, we refer to Section 5.2 in \cite{BK.Kad_Rin.1997.OpAlg_I}. Let $\pi:M\longrightarrow\BII(H)$ be faithful normal unital $^{*}$-representation. Following Theorem 5.2.2 in \cite{BK.Kad_Rin.1997.OpAlg_I}, get unique resolution of the identity in $\BII(H)$ associated to $x$. It determines unique spectral measure $E_{x,M}:\mathfrak{B}(\mathbb{R})\longrightarrow\BII(H)$. Pulled-back along $\pi$, uniqueness implies $E_{x,M}$ is independent of our choice of faithful normal $^{*}$-representation.\par
Let $\mathcal{N}_{x,M}:=\lset{}Z\in\mathfrak{B}(\mathbb{R})\ \vset\ E_{x,M}(Z)=0\rset$. Intersecting with $\specM x$ shows $\mathcal{N}_{x,M}$ is a $\sigma$-ideal of null sets of the Borel $\sigma$-algebra $\mathfrak{B}\lc\specM x\rc$. Following the construction in Example \ref{BSP.Wstar}, get $W^{*}$-algebra $L^{\infty}\lc\specM x,\mathcal{N}_{x,M}\rc$ s.t.~$C\lc\specM x\rc\subset L^{\infty}\lc\specM x,\mathcal{N}_{x,M}\rc$ is $\sigma$-weakly dense. This shows $1)$. For $2)$, see \cite{BK.Kad_Rin.1997.OpAlg_I}. Get $3)$ by Lemma \ref{LEM.Cstar_FC}.
\end{proof}

\begin{dfn}\label{DFN.Wstar_FC}
Let $M$ be a $W^{*}$-algebra. For all normal $x\in M$, we call

\begin{itemize}
\item[1)] $\Gamma_{x,M}$ as in Equation \ref{EQ.LEM.Wstar_FC_1} the bounded measurable functional calculus of $x$ in $M$,

\item[2)] $W_{M}^{*}(x)$ as in Equation \ref{EQ.LEM.Wstar_FC_1} the $W^{*}$-algebra generated by $x$ in $M$.
\end{itemize}
\end{dfn}

\begin{ntn}\label{NTN.Wstar_FC}
Unless stated otherwise, we suppress $W^{*}$-algebras in subscripts of spectral measures, spectra, bounded measurable functional calculus and generated\linebreak $W^{*}$-algebras. We extend to measurable operators in Notation \ref{NTN.Wstar_CLRA_FC}.
\end{ntn}


\subsubsection*{Functional calculus of self-adjoint unbounded operators}

Let $H$ be Hilbert space. For all normal $T\in\BII(H)$, the map $Z\mapsto E_{T}(Z):=\chi_{Z}(T)$ defined on $\mathfrak{B}(\mathbb{C})$ is spectral measure on $\mathbb{C}$ with values in $\BII(H)$. We extend to self-adjoint unbounded operators.\par
Let $T\in\UBII(H)_{h}$. We call $B(T):=T\lc{}1+T^{2}\rc^{-\frac{1}{2}}\in\BII(H)$ its bounded transform \cite{BK.Sch.2012.Unbounded_Operators}. We have $\spec B_{T}\subset [-1,1]$. For all $t\in [-1,1]$, set 

\begin{align}\label{EQ.SSEC.A_Fnd_FC_4}
\varphi(t):=t\lc{}1-t^{2}\rc^{-\frac{1}{2}}.
\end{align}

\noindent Note $\varphi$ is $E_{B(T)}$-a.e.~finite measurable and invertible on $[-1,1]$. Formally, $B(T)=\varphi^{-1}(T)$ by change of variable $t\mapsto T$ in Equation \ref{EQ.SSEC.A_Fnd_FC_4}. For all $Z\in\mathfrak{B}(\mathbb{R})$, set

\begin{align}\label{EQ.SSEC.A_Fnd_FC_5}
E_{T}(Z):=E_{B(T)}\lc\varphi^{-1}(Z)\rc{}.
\end{align}

\noindent Equation \ref{EQ.SSEC.A_Fnd_FC_5} defines spectral measure $E_{T}:\mathfrak{B}(\mathbb{R})\longrightarrow\BII(H)$ \cite{BK.Sch.2012.Unbounded_Operators}.

\begin{dfn}\label{DFN.FC_Unbd}
Let $T\in\UBII(H)_{h}$.

\begin{itemize}
\item[1)] We call $E_{T}$ the spectral measure of $T$.

\item[2)] For all $g\in\SII\lc{}E_{T}\rc$, set

\begin{align}\label{EQ.FC_Unbd_1}
\Gamma_{T}(g):=g(T):=\mathrm{I}_{E_{T}}(g)=\int gdE_{T}.
\end{align}

\begin{reapply}
\end{reapply}

\item[3)] We call $\Gamma_{T}:\SII\lc{}E_{T}\rc\longrightarrow\UBII(H)$ the functional calculus of $T$.
\end{itemize}
\end{dfn}

\begin{rem}\label{REM.FC_BT_IS_CT}
Note $B(T)$ is denoted by $Z_{T}$ in \cite{BK.Sch.2012.Unbounded_Operators}. Instead of $B(T)$, \cite{BK.Ped.1989.Analysis_Now} uses the Cayley transform $C(T):=\frac{T-i}{T+i}\in\BII(H)$. The induced spectral measure is $E_{T}$. Thus $B(T)$ and $C(T)$ define identical spectral measure, hence functional calculus.
\end{rem}

Theorem 5.9 and Proposition 5.10 in \cite{BK.Sch.2012.Unbounded_Operators} collect elementary properties of functional calculus. The spectral theorem for self-adjoint unbounded operators further shows each $E_{T}:\mathfrak{B}(\mathbb{R})\longrightarrow\BII(H)$ is the unique spectral measure s.t.~$T=\mathrm{I}_{E_{T}}\lc\id_{\mathbb{R}}\rc{}=\int tdE_{T}$.\par
Functional calculus restricts to bounded measurable functional calculus. This uses spectra of densely defined operators. For self-adjoint unbounded operators, spectra are the support of spectral measures. Definition \ref{DFN.Cstar_FC} is subsumed if we are given faithful normal unital $^{*}$-representation. Unitality is necessary.

\begin{dfn}
Let $T$ be a densely defined closable operator on $H$. Its resolvent set is $\rsl T:=\lset\lambda\in\mathbb{C}\ \vset\ \lc{}T-\lambda I\rc^{-1}\in\BII(H)\rset$ and its spectrum is $\spec T:=\mathbb{C}\setminus\rsl T$.
\end{dfn}

\begin{rem}
For all faithful normal unital $^{*}$-representation $\pi:M\longrightarrow\BII(H)$ of a $W^{*}$-algebra $M$, get $\specM x=\spec\pi(x)$ for all normal $x\in M$. 
\end{rem}

If $T\in\UBII(H)_{h}$, then $\spec T\subset\mathbb{R}$ and $\pm i\in\rsl T$. For all $g\in C\lc\spec T\rc$, get $\spec g(T)=\overline{g\lc\spec T\rc{}}\subset\mathbb{R}$ by the spectral mapping theorem \lc{}cf.~Proposition 5.25 in \cite{BK.Sch.2012.Unbounded_Operators}\rc{}. If moreover $g,g^{-1}\in C\lc\spec T\rc$, then $\spec g(T)=g\lc\spec T\rc$.

\begin{prp}
If $T\in\UBII(H)_{h}$, then $\supp E_{T}=\spec T\subset\mathbb{R}$.
\end{prp}
\begin{proof}
Proposition 5.10 in \cite{BK.Sch.2012.Unbounded_Operators}.
\end{proof}

\begin{dfn}\label{DFN.Resolvent}
Let $T\in\UBII(H)_{h}$.

\begin{itemize}
\item[1)] Let $a\in\mathbb{C}$. For all $z\in\mathbb{C}\setminus\lset{}a\rset$, set 

\begin{align}\label{EQ.DFN.Resolvent_1}
R_{a}(z):=\big(z-a\big)^{-1}.    
\end{align}

\noindent If $a\in\rsl T$, then $R_{a}(T)\in\BII(H)$ is the resolvent of $T$ in $a$.

\item[2)] Set $L^{\infty}\lc\spec T,dE_{T}\rc{}:=L^{\infty}\lc\spec T,\mathcal{N}\lc{}E_{T}\rc\rc$.
\end{itemize} 
\end{dfn}

\begin{ntn}\label{NTN.Resolvents}
For all $T\in\UBII(H)_{h}$, let $R_{\pm i}(T)$ denote both $R_{i}(T)$ or $R_{-i}(T)$. Note $\pm i\in\mathbb{C}\setminus\mathbb{R}$ lies in the resolvent set of all self-adjoint unbounded operators.
\end{ntn}


\pagebreak


Let $T\in\UBII(H)_{h}$ and $g\in\SII\lc{}E_{T}\rc$. Bounding sequences let us write $g(T)$ as pointwise $\|.\|_{H}$-limit. For all $Z\in\mathfrak{B}(\mathbb{R})$, note Equation \ref{EQ.SSEC.A_Fnd_FC_1} ensures $g(T)E_{T}(Z)=\lc{}g\chi_{Z}\rc{}(T)$. For all $u\in H$, we have $u\in\dom g(T)$ if and only if

\begin{align}\label{EQ.SSEC.A_Fnd_FC_6}
g(T)(u)=\|.\|_{H}\textrm{-}\lim_{k\in\mathbb{N}}\hspace{0.025cm} g(T)E_{T}\lc{}Z_{k}\rc{}(u)
\end{align}

\noindent exists for a bounding sequence $\lset{}Z_{k}\rset_{k\in\mathbb{N}}$ of $g$. In fact, Equation \ref{EQ.SSEC.A_Fnd_FC_6} is Equation \ref{EQ.SSEC.A_Fnd_FC_3} for spectral integration w.r.t.~$E_{T}$. If $u\in\dom g(T)$, then $g(T)(u)$ exists and is independent of choice of bounding sequence of $g$ by Theorem 4.13 in \cite{BK.Sch.2012.Unbounded_Operators}. 

\begin{prp}\label{PRP.FC_BT_IS_CT}
For all $T\in\UBII(H)_{h}$, $W^{*}\lc{}B(T)\rc{}=W^{*}\lc{}C(T)\rc{}=W^{*}\lc{}R_{\pm i}(T)\rc$.
\end{prp}
\begin{proof}
Following Remark \ref{REM.FC_BT_IS_CT}, we know $W^{*}\lc{}B(T)\rc{}=W^{*}\lc{}C(T)\rc$. Since we further have $R_{\pm i}\in L^{\infty}\lc\spec T,dE_{T}\rc$, get $R_{\pm i}(T)\in W^{*}\lc{}C(T)\rc$ by Theorem 5.3.8 in \cite{BK.Ped.1989.Analysis_Now}. We directly verify $C(T)=R_{i}(T)R_{-i}(T)$ and get $W^{*}\lc{}C(T)\rc{}=W^{*}\lc{}R_{i}(T),R_{-i}(T)\rc{}=:W^{*}\lc{}R_{\pm i}(T)\rc$.
\end{proof}

\begin{dfn}\label{DFN.FC_Generated}
For all $T\in\UBII(H)_{h}$, we call $W^{*}(T):=W^{*}\lc{}B(T)\rc$ the $W^{*}$-algebra generated by $T$.
\end{dfn}

\begin{rem}\label{REM.FC_Generated}
If $T\in\BII(H)_{h}$, then $W^{*}(T)=W_{\BII(H)}^{*}(T)=W^{*}\lc{}T,I_{H}\rc$.
\end{rem}

If $T\in\UBII(H)_{h}$, then $\Gamma_{T}$ restricts to $L^{\infty}\lc\spec T,dE_{T}\rc$. Proposition \ref{PRP.FC_Bd} uses the latter to formulate bounded measurable functional calculus.

\begin{prp}\label{PRP.FC_Bd}
Let $T\in\UBII(H)_{h}$.

\begin{itemize}
\item[1)] We have normal unital $^{*}$-isomorphism $\Gamma_{T}:L^{\infty}\lc\spec T,dE_{T}\rc\longrightarrow W^{*}(T)$.

\item[2)] If $M\subset\BII(H)$ is a $W^{*}$-subalgebra s.t.~$E_{T}(Z)\in M$ for all $Z\in\mathfrak{B}(\mathbb{R})$, then $W^{*}(T)\subset M$.
\end{itemize}
\end{prp}
\begin{proof}
Since $W^{*}(T)=W^{*}\lc{}C(T)\rc$ by Proposition \ref{PRP.FC_BT_IS_CT}, we use the functional calculus in \cite{BK.Ped.1989.Analysis_Now}. $W^{*}(T)=W^{*}(T)''$ by Proposition \ref{PRP.Wstar_Equivalence}. Since $T$ is self-adjoint, Lemma 5.2.8 and Theorem 5.3.8 in \cite{BK.Ped.1989.Analysis_Now} therefore show $1)$. In the setting of $2)$, $P\lc{}W^{*}(T)\rc\subset P(M)$ and Proposition \ref{PRP.Wstar_Generated} imply $W^{*}(T)=W^{*}\lc{}P\lc{}W^{*}(T)\rc\subset W^{*}\lc{}P(M)\rc{}=M$.
\end{proof}

\begin{rem}\label{REM.FC_Bd}
If $\pi:M\longrightarrow\BII(H)$ is a faithful normal unital $^{*}$-representation of a $W^{*}$-algebra $M$, then $\pi\circ\Gamma_{x,M}=\Gamma_{\pi(x)}$ for all $x\in M_{h}$. Unitality is necessary.
\end{rem}

Bounded measurable functional calculus lets us test for injectivity by considering the mass of $\lset{}0\rset$ under $E_{T}$ as per Remark \ref{REM.FC_Injectivity}.

\begin{rem}\label{REM.FC_Injectivity}
Following Remark \ref{REM.SpecInt_Identities}, note $E_{T}\lc\lset{}0\rset\rc{}=\chi_{\lset{}0\rset}(T)=\delta_{0}(T)$ is the Hilbert space projection onto $\ker T$ since $u\in\ker T$ if and only if $\supp E_{T}^{u}=\lset{}0\rset$ for all $u\in H$.
\end{rem}

\begin{prp}\label{PRP.FC_Injectivity}
If $T\in\UBII(H)_{h}$, then $T$ is injective if and only if $E_{T}\lc\lset{}0\rset\rc{}=0$. 
\end{prp}
\begin{proof}
If $T$ is injective, then get $E_{T}\lc\lset{}0\rset\rc{}=0$ as per Remark \ref{REM.FC_Injectivity}. If $E_{T}\lc\lset{}0\rset\rc{}=0$, then $t\mapsto t^{-1}$ is $E_{T}$-a.e.~finite. Thus $T^{-1}$ is densely defined closed by functional calculus, hence $T$ is injective if $E_{T}\lc\lset{}0\rset\rc{}=0$.
\end{proof}


\pagebreak


We use Lemma \ref{LEM.FC_Unitary_Com} to directly verify affiliation with $W^{*}$-algebras. Lemma \ref{LEM.FC_Preservation_I} and Lemma \ref{LEM.FC_Preservation_II} provide necessary and sufficient conditions for preserving bounded measurable functional calculus.

\begin{lem}\label{LEM.FC_Unitary_Com}
If $T\in\UBII(H)_{h}$ and $U\in\UII\lc\BII(H)\rc$, then we have $TU=UT$ if and only if $\lb{}E_{T}(Z),U\rb{}=0$ for all $Z\in\mathfrak{B}(\mathbb{R})$.
\end{lem}
\begin{proof}
Assume $TU=UT$. Proposition \ref{PRP.FC_BT_IS_CT} shows $\lb{}R_{\pm i}(T),U\rb{}=0$ yields $\lb{}E_{T}(Z),U\rb{}=0$ in each case. Note $T=U^{*}TU$ and $U\dom T\subset\dom T$ imply $R_{\pm i}(T)=U^{*}R_{\pm i}(T)U$. Thus $TU=UT$ implies $\lb{}E_{T}(Z),U\rb{}=0$ for all $Z\in\mathfrak{B}(\mathbb{R})$.\par
Assume $\lb{}E_{T}(Z),U\rb{}=0$, i.e.~$E_{T}(Z)=UE_{T}U^{*}$ for all $Z\in\mathfrak{B}(\mathbb{R})$. Thus $E_{T}^{v}=E_{T}^{U^{*}v}$ for all $v\in H$, hence $v\in\dom T$ if and only if $U^{*}v\in\dom T$. Get $\dom T=\dom TU^{*}$. We also know $\lb{}g(T),U\rb{}=0$ for all $g(T)\in W^{*}(T)$ since $W^{*}(T)$ is generated by all $E_{T}(Z)$ for all $Z\in\mathfrak{B}(\mathbb{R})$ by Proposition \ref{PRP.Wstar_Generated} and Proposition \ref{PRP.FC_Bd}. The spectral theorem and Equation \ref{EQ.SSEC.A_Fnd_FC_6} imply $w\in\dom T$ if and only if there exists bounding sequence $\lset{}Z_{k}\rset_{k\in\mathbb{N}}$ of $\id_{\mathbb{R}}$ s.t.~

\begin{align}\label{EQ.LEM.FC_Unitary_Com_1}
T(w)=\int tdE_{T}^{w}=\|.\|_{H}\textrm{-}\lim_{n\in\mathbb{N}}\hspace{0.025cm} g_{n}(T)(w).
\end{align}

\noindent If $w\in\dom T$, then the limit in Equation \ref{LEM.FC_Unitary_Com} exists and is independent of choice of bounding sequence of $\id_{\mathbb{R}}$. For all $v\in\dom T=\dom TU^{*}$, we see Equation \ref{EQ.LEM.FC_Unitary_Com_1} implies $T(v)=U\lc\|.\|_{H}$-$\lim_{n\in\mathbb{N}}g_{n}(T)U^{*}(v)\rc{}=UTU^{*}(v)$. Thus $T=UTU^{*}$, hence $TU=UT$.
\end{proof}

\begin{dfn}\label{DFN.FC_Preservation}
Let $H_{0}$ and $H_{1}$ be Hilbert spaces. Let $M\subset\BII\lc{}H_{0}\rc$ be $W^{*}$-algebra and $\phi:M\longrightarrow\BII\lc{}H_{1}\rc$ normal unital $^{*}$-homomorphism. If $T\in\UBII\lc{}H_{0}\rc_{h}$ s.t.~$\im E_{T}\subset M$, then\linebreak we define the push-forward spectral measure $\phi\lc{}E_{T}\rc$ of $T$ under $\phi$ by setting

\begin{align}\label{EQ.DFN.FC_Preservation_1}
\phi\lc{}E_{T}\rc(Z):=\phi\lc{}E_{T}(Z)\rc{}
\end{align}

\noindent for all $Z\in\mathfrak{B}(\mathbb{R})$.
\end{dfn}

Equation \ref{EQ.DFN.FC_Preservation_1} defines spectral measure $\phi\lc{}E_{T}\rc{}:\mathfrak{B}(\mathbb{R})\longrightarrow\BII\lc{}H_{1}\rc$ if we are in the setting of Definition \ref{DFN.FC_Preservation}. Lemma \ref{LEM.FC_Preservation_II} shows push-forward spectral measures link bounded measurable functional calculus across Hilbert spaces.

\begin{lem}\label{LEM.FC_Preservation_I}
Let $H_{0}$ and $H_{1}$ be Hilbert spaces. Let $M\subset\BII\lc{}H_{0}\rc$ be $W^{*}$-algebra and $\phi:M\longrightarrow\BII\lc{}H_{1}\rc$ normal unital $^{*}$-homomorphism. If $T_{0}\in\UBII\lc{}H_{0}\rc_{h}$ and $T_{1}\in\UBII\lc{}H_{1}\rc_{h}$ s.t.~$\im E_{T}\subset M$ and $\phi\lc{}g\lc{}T_{0}\rc\rc{}=g\lc{}T_{1}\rc$ for all $g\in C_{c}(\mathbb{R})$, then $\phi\lc{}E_{T_{0}}\rc{}=E_{T_{1}}$.
\end{lem}
\begin{proof}
Let $\lambda\in\mathbb{R}$, and $\lset{}g_{n}^{\lambda}\rset_{n\in\mathbb{N}}\subset C_{b}(\mathbb{R})$ s.t.~$\sup_{n\in\mathbb{N}}\dblv{}g_{n}^{\lambda}\dblv{}<\infty$ and $\chi_{\lc{}-\infty,\lambda\rb{}}(t)=\lim_{n\in\mathbb{N}}g_{n}^{\lambda}(t)$ for all $t\in\mathbb{R}$. For all self-adjoint $S$ on arbitrary Hilbert space $H$ and $u\in H$, $E_{S}^{u}$ is finite and we have

\begin{align}\label{EQ.LEM.FC_Preservation_I_1}
\dblv{}\lc{}E_{S}\lc\lc{}-\infty,\lambda\rb\rc{}-g_{n}^{\lambda}(S)\rc{}(u)\dblv_{H}^{2}=\int_{\mathbb{R}}\lc\chi_{\lc{}-\infty,\lambda\rb{}}(t)-g_{n}^{\lambda}(t)\rc^{2}dE_{S}^{u}
\end{align}

\noindent by functional calculus. Thus $E_{S}\lc\lc{}-\infty,\lambda\rb\rc{}=\s$-$\lim_{n\in\mathbb{N}}g_{n}^{\lambda}(S)$ by dominated convergence.\par


\pagebreak


Let $\{\varphi_{n}\}_{n\in\mathbb{N}}\subset C_{c}(\mathbb{R})$ s.t.~$0\leq\varphi_{n}\leq\varphi_{n+1}\leq 1$ for all $n\in\mathbb{N}$ and $\lim_{n\in\mathbb{N}}\varphi_{n}(t)=1$ for all $t\in\mathbb{R}$. Arguing as for Equation \ref{EQ.LEM.FC_Preservation_I_1} shows $\s$-$\lim_{n\in\mathbb{N}}g_{n}^{\lambda}(S)=\s$-$\lim_{n\in\mathbb{N}}\lc{}g_{n}^{\lambda}\varphi_{n}\rc{}(S)$. We therefore assume $\{g_{n}\}_{n\in\mathbb{N}}\subset C_{c}(\mathbb{R})$ in Equation \ref{EQ.LEM.FC_Preservation_I_1} without loss of generality.\par
If $\phi\lc{}E_{T_{0}}\lc\lc{}-\infty,\lambda\rb\rc\rc{}=E_{T_{1}}\lc\lc{}-\infty,\lambda\rb\rc$ for all $\lambda\in\mathbb{R}$, then $\phi\lc{}E_{T_{0}}\rc{}=E_{T_{1}}$ as resolutions of the identity are unique. We show the former by approximating in strong operator topology as above. For fixed but arbitrary $\lset{}g_{n}^{\lambda}\rset_{n\in\mathbb{N}}\subset C_{c}(\mathbb{R})$ and for all $n\in\mathbb{N}$, note Equation \ref{EQ.LEM.FC_Preservation_I_1} holds uniformly for all self-adjoint unbounded operators. Sequential strong continuity of $\phi$ therefore implies

\begin{align*}
\phi\lc{}E_{T_{0}}\lc\lc{}-\infty,\lambda\rb\rc\rc{} & = \s\textrm{-}\lim_{n\in\mathbb{N}} \hspace{0.025cm} \phi\lc{}g_{n}\lc{}T_{0}\rc\rc \phantom{\bigg)} \\
& = \s\textrm{-}\lim_{n\in\mathbb{N}}\hspace{0.025cm} g_{n}\lc{}T_{1}\rc \phantom{\bigg)} \\
& = E_{T_{1}}\lc\lc{}-\infty,\lambda\rb\rc \phantom{\bigg)}
\end{align*}

\noindent for all $\lambda\in\mathbb{R}$. The above calculation uses Remark \ref{REM.Wstar_Con}.
\end{proof}

\begin{lem}\label{LEM.FC_Preservation_II}
Let $H_{0}$ and $H_{1}$ be Hilbert spaces. Let $T_{0}\in\UBII\lc{}H_{0}\rc_{h}$ and $T_{1}\in\UBII\lc{}H_{1}\rc_{h}$. If $\phi:W^{*}\lc{}T_{0}\rc\longrightarrow W^{*}\lc{}T_{1}\rc$ is a normal unital $^{*}$-homomorphism s.t.~$\phi\lc{}E_{T_{0}}\rc{}=E_{T_{1}}$, then 

\begin{itemize}
\item[1)] $\spec T_{1}\subset\spec T_{0}$ and $\mathcal{N}\lc{}E_{T_{0}}\rc\subset\mathcal{N}\lc{}E_{T_{1}}\rc$,

\item[2)] $\phi\lc{}g\lc{}T_{0}\rc\rc{}=g\lc{}T_{1}\rc$ for all $g\in L^{\infty}\lc\spec T_{0},dE_{T_{0}}\rc$,

\item[3)] we have commutative diagram of normal unital surjective $^{*}$-homomorphisms

\begin{equation}\label{EQ.LEM.FC_Preservation_II_1}
\begin{tikzcd}
L^{\infty}\lc\spec T_{0},dE_{T_{0}}\rc\arrow[rr,"\Gamma_{T_{0}}"]\arrow[dd,"\res"] & & W^{*}\lc{}T_{0}\rc\arrow[dd,"\phi"] \\
& & \\
L^{\infty}\lc\spec T_{1},dE_{T_{1}}\rc\arrow[rr,"\Gamma_{T_{1}}"] & & W^{*}\lc{}T_{1}\rc{}
\end{tikzcd}
\end{equation}

\begin{reapply}
\end{reapply}

\noindent with $\res$ the restriction map given by $\spec T_{1}\subset\spec T_{0}$.
\end{itemize} 
\end{lem}
\begin{proof}
We directly verify $\mathcal{N}\lc{}E_{T_{0}}\rc\subset\mathcal{N}\lc{}E_{T_{1}}\rc$. Since we have $\lambda\in\spec T_{1}=\supp E_{T_{1}}$ if and only if $\phi\lc{}E_{T_{0}}\lc{}N_{\lambda}\rc\rc{}=E_{T_{1}}\lc{}N_{\lambda}\rc\neq 0$ for all open neighbourhoods $N_{\lambda}$ of $\lambda$, get $1)$ at once. If $g\in\SII\lc{}E_{T_{0}}\rc$ is $E_{T_{0}}$-a.e.~bounded, then $g\in\SII\lc{}E_{T_{1}}\rc$ is $E_{T_{1}}$-a.e.~bounded.\par
For all $n\in\mathbb{N}$, let $\lset{}Z_{k,m}\rset_{k,m\in\mathbb{N}}\subset\mathfrak{B}(\mathbb{R})$ s.t.~following pointwise $E_{T_{0}}$-a.e.~approximation of $\id_{\mathbb{R}}\chi_{[-n,n]}$ on $\spec T_{0}$ holds. For $E_{T_{0}}$-a.e.~$t\in\spec T$, get $\lset{}a_{k,m}\rset_{k,m\in\mathbb{N}}\subset\mathbb{R}$ and

\begin{align}\label{EQ.LEM.FC_Preservation_II_2}
t\chi_{[-n,n]}(t)=\lim_{m\in\mathbb{N}}\hspace{0.025cm} \sum_{k=1}^{m}a_{k,m}\chi_{Z_{k,m}}(t).
\end{align}

\noindent Using $1)$, Equation \ref{EQ.LEM.FC_Preservation_II_2} further yields $E_{T_{1}}$-a.e.~approximation of $\id_{\mathbb{R}}\chi_{[-n,n]}$ on $\spec T_{1}$. The approximations we use here are uniformly bounded, hence yield bounded strong limits upon evaluation using $T_{0}$, resp.~$T_{1}$.\par


\pagebreak


Normality and $\phi\lc{}E_{T_{0}}\rc{}=E_{T_{1}}$ yield $\phi\lc{}T_{0}E_{T_{0}}\lc{}[-n,n]\rc\rc{}=T_{1}E_{T_{1}}\lc{}[-n,n]\rc$ for all $n\in\mathbb{N}$. For all $g\in L^{\infty}\lc\spec T_{0},dE_{T_{0}}\rc$, we see $1)$ implies $g\in L^{\infty}\lc\spec T_{1},dE_{T_{1}}\rc$ upon restriction. Strong convergence of sequences further implies

\begin{align}\label{EQ.LEM.FC_Preservation_II_3}
g\lc{}T_{0}\rc{}=\s\textrm{-}\lim_{n\in\mathbb{N}}\hspace{0.025cm} g\lc{}T_{0}E_{T_{0}}\lc{}[-n,n]\rc\rc{},\ g\lc{}T_{1}\rc{}=\s\textrm{-}\lim_{n\in\mathbb{N}}\hspace{0.025cm} g\lc{}T_{1}E_{T_{1}}\lc{}[-n,n]\rc\rc{}.
\end{align}

\noindent If $\phi\lc{}g\lc{}T_{0}E_{T_{0}}\lc{}[-n,n]\rc\rc{}=g\lc{}T_{1}E_{T_{1}}\lc{}[-n,n]\rc\rc$ for all $g\lc{}T_{0}\rc\in W^{*}\lc{}T_{0}\rc$ and $n\in\mathbb{N}$, then $2)$ holds by Equation \ref{EQ.LEM.FC_Preservation_II_3}. Ergo $2)$, and thereby $3)$, reduces to the bounded case.\par
Assume $T_{0}$ and $T_{1}$ are bounded. Thus $\phi\lc{}T_{0}\rc{}=T_{1}$, hence $g\lc\phi\lc{}T_{0}\rc\rc{}=g\lc{}T_{1}\rc$ for all $g\in C\lc\spec T_{k}\rc$. For all $k\in\lset{}0,1\rset$, Proposition \ref{PRP.FC_BT_IS_CT} shows $R_{\pm i}\in C\lc\spec T_{k}\rc$ and $W^{*}\lc{}T_{k}\rc{}=W^{*}\lc{}R_{\pm i}\lc{}T_{k}\rc\rc$. Normality implies $g\lc\phi\lc{}T_{0}\rc\rc{}=g\lc{}T_{1}\rc$ for all $g\in L^{\infty}\lc\spec T_{0},dE_{T_{0}}\rc$. Get $2)$. Using the latter, we directly verify $3)$. The general case follows as discussed above.
\end{proof}

\begin{cor}\label{COR.FC_Preservation}
Let $H_{0}$ and $H_{1}$ be Hilbert spaces. Let $M\subset\BII\lc{}H_{0}\rc$ and $N\subset\BII\lc{}H_{1}\rc$ be $W^{*}$-algebras. We consider $x\in M_{h}$ and $y\in N_{h}$. If $\phi:W_{M}^{*}(x)\longrightarrow W_{N}^{*}(y)$ is a normal unital $^{*}$-homomorphism s.t.~$\phi(x)=y$, then $\phi\lc{}E_{x,M}\rc{}=E_{y,N}$ and Lemma \ref{LEM.FC_Preservation_II} applies to $\phi$.
\end{cor}
\begin{proof}
Note $\supp E_{x,M}=\specM x$ and $\supp E_{y,N}=\specN y$. Let $\specM x,\specN y\subset K$ for compact $K\subset\mathbb{R}$. Lemma \ref{LEM.FC_Preservation_I} shows $\phi\lc{}g(x)\rc{}=g(y)$ for all $g\in C(K)$ suffices. We reduce to polynomials by approximating in norms. The $^{*}$-homomorphism property concludes.
\end{proof}


\subsubsection*{Joint functional calculus of strongly commuting self-adjoint unbounded operators}

Let $H$ be a Hilbert space. Let $T,S\in\UBII(H)_{h}$. If $\lb{}E_{T}\lc{}Z_{0}\rc{},E_{S}\lc{}Z_{1}\rc\rb{}=0$ for all $Z_{0},Z_{1}\in\mathfrak{B}(\mathbb{R})$, then we determine joint spectral measure by setting

\begin{align}\label{EQ.SSEC.A_Fnd_FC_7}
E_{T,S}\lc{}Z_{0}\times Z_{1}\rc{}:=E_{T}\lc{}Z_{0}\rc{}E_{S}\lc{}Z_{1}\rc{}
\end{align}

\noindent for all $Z_{0},Z_{1}\in\mathfrak{B}(\mathbb{R})$. Equation \ref{EQ.SSEC.A_Fnd_FC_7} defines spectral measure $E_{T,S}:\mathfrak{B}\lc\mathbb{R}\times\mathbb{R}\rc\longrightarrow\BII(H)$ by Theorem 4.10 in \cite{BK.Sch.2012.Unbounded_Operators}.

\begin{dfn}\label{DFN.JFC_Unbd}
Let $T,S\in\UBII(H)_{h}$. We say that $T$ and $S$ commute strongly if $\lb{}E_{T}\lc{}Z_{0}\rc{},E_{S}\lc{}Z_{1}\rc\rb{}=0$ for all $Z_{0},Z_{1}\in\mathfrak{B}(\mathbb{R})$. Assume $T$ and $S$ commute strongly.

\begin{itemize}
\item[1)] We call $E_{T,S}$ the joint spectral measure of $T$ and $S$.

\item[2)] For all $g\in\SII\lc{}E_{T,S}\rc$, set

\begin{align}\label{EQ.JFC_Unbd_1}
\Gamma_{T,S}(g):=g\lc{}T,S\rc{}:=\mathrm{I}_{E_{T,S}}(g)=\int gdE_{T,S}.
\end{align}

\begin{reapply}
\end{reapply}

\item[3)] We call $\Gamma_{T,S}:\SII\lc{}E_{T,S}\rc\longrightarrow\UBII(H)$ the joint functional calculus of $T$ and $S$.
\end{itemize}
\end{dfn}

\begin{rem}\label{REM.JFC_Strong_Com}
The commutator $[\hspace{-0.025cm}\blank\hspace{0.038725cm} ,\hspace{-0.055cm} \blank]:\BII(H)\times\BII(H)\longrightarrow\BII(H)$ in $\BII(H)$ is given by $[T,S]:=TS-ST$ for all $T,S\in\BII(H)$. It is separately continuous in strong operator topology. If $M,N\subset\BII(H)$ are $W^{*}$-algebras with strongly dense subsets $M_{0}\subset M$ and $N_{0}\subset N$ s.t.~$\lb{}M_{0},N_{0}\rb{}=0$, then $\lb{}M,N\rb{}=0$ by separate strong continuity.
\end{rem}

\begin{prp}\label{PRP.JFC_Strong_Com}
For all $T,S\in\UBII(H)_{h}$, the following are equivalent:

\begin{itemize}
\item[1)] $T$ and $S$ commute strongly,

\item[2)] $\lb{}R_{a}(T),R_{b}(S)\rb{}=0$ for $a\in\rsl T$ and $b\in\rsl S$,

\item[3)] $\lb{}g(T),h(S)\rb{}=0$ for all $g\in L^{\infty}\lc\spec T,dE_{T}\rc$ and $h\in L^{\infty}\lc\spec S,dE_{S}\rc$,

\item[4)] $\lb{}B(T),B(S)\rb{}=0$,

\item[5)] $\lb{}C(T),C(S)\rb{}=0$.
\end{itemize}
\end{prp}
\begin{proof}
Equivalence of $1)$ and $2)$ is Proposition 5.27 in \cite{BK.Sch.2012.Unbounded_Operators}. Continuity of commutators as per Remark \ref{REM.JFC_Strong_Com} ensures Proposition \ref{PRP.FC_BT_IS_CT} shows equivalence of $2)$ to $5)$.
\end{proof}

\begin{prp}
If $T,S\in\UBII(H)_{h}$ commute strongly, then

\begin{itemize}
\item[1)] $E_{T,S}$ is the unique spectral measure s.t.~$T=\int tdE_{T,S}$ and $S=\int sdE_{T,S}$,

\item[2)] $\supp E_{T,S}\subset\spec T\times\spec S$.
\end{itemize}
\end{prp}
\begin{proof}
Get $1)$ by Lemma 5.22 in \cite{BK.Sch.2012.Unbounded_Operators}. Thus $E_{T,S}$ is joint spectral measure given in the proof of Theorem 5.23 in \cite{BK.Sch.2012.Unbounded_Operators}, hence $2)$ follows by Proposition 5.24 in \cite{BK.Sch.2012.Unbounded_Operators}.
\end{proof}

\begin{rem}
Note $\supp E_{T,S}\neq\spec T\times\spec S$ in general as $\BII(H)$ has zero divisors. Inequality therefore occurs if $E_{S}\lc{}N_{s}\rc{}H\subset \lc{}E_{T}\lc{}N_{t}\rc{}H\rc^{\perp}$ for a product open neighbourhood $N_{t}\times N_{s}$ of $(t,s)\in\spec T\times\spec S$.
\end{rem}

Let $M,N\subset\BII(H)$ be commutative $W^{*}$-subalgebras s.t.~$W^{*}\lc{}M,N\rc\subset\BII(H)$ is also commutative $W^{*}$-subalgebra. We determine normal unital injective $^{*}$-homomorphism $M\otimes N\longrightarrow\BII(H)$ by mapping

\begin{align}\label{EQ.SSEC.A_Fnd_FC_8}
x\otimes y\mapsto xy=yx
\end{align}

\noindent for all $x\in M$ and $y\in N$. Get $W^{*}$-subalgebra $M\otimes N\subset\BII(H)$. Proposition \ref{PRP.JFC_Bd} thereby extends Proposition \ref{PRP.FC_Bd} using jointly generated $W^{*}$-algebras. 

\begin{dfn}
Let $T,S\in\UBII(H)_{h}$ commute strongly.

\begin{itemize}
\item[1)] The joint spectrum of $T$ and $S$ is $\spec T\times S:=\supp E_{T,S}$.

\item[2)] Set $L^{\infty}\lc\spec T\times S,dE_{T,S}\rc{}:=L^{\infty}\lc\spec T\times S,\mathcal{N}\lc{}E_{T,S}\rc\rc$.

\item[3)] We call $W^{*}\lc{}T,S\rc{}:=W^{*}(T)\otimes W^{*}(S)$ the $W^{*}$-algebra generated by $T$ and $S$.
\end{itemize} 
\end{dfn}


\pagebreak


If $T,S\in\UBII(H)_{h}$ commute strongly, then $\Gamma_{T,S}$ restricts to $L^{\infty}\lc\spec T\times S,dE_{T,S}\rc$ as in the case of one self-adjoint unbounded operator. Proposition \ref{PRP.JFC_Bd} uses the latter to formulate bounded measurable joint functional calculus.

\begin{prp}\label{PRP.JFC_Bd}
Let $T,S\in\UBII(H)_{h}$ commute strongly.

\begin{itemize}
\item[1)] We have normal unital $^{*}$-isomorphism $\Gamma_{T,S}:L^{\infty}\lc\spec T\times S,dE_{T,S}\rc\longrightarrow W^{*}\lc{}T,S\rc$.

\item[2)] If $M\subset\BII(H)$ is a $W^{*}$-subalgebra s.t.~$E_{T,S}\lc{}Z_{0}\times Z_{1}\rc\in M$ for all $Z_{0},Z_{1}\in\mathfrak{B}(\mathbb{R})$, then $W^{*}\lc{}T,S\rc\subset M$.
\end{itemize}
\end{prp}
\begin{proof}
Note $L^{\infty}\lc\spec T\times S,dE_{T,S}\rc{}=L^{\infty}\lc\spec T\times\spec S,dE_{T,S}\rc$ by construction of joint spectral measures, and $L^{\infty}\lc\spec T\times\spec S,dE_{T,S}\rc\cong L^{\infty}\lc\spec T,dE_{T}\rc\otimes L^{\infty}\lc\spec S,dE_{S}\rc$ naturally. All claims reduce to elementary tensors. Apply Proposition \ref{PRP.FC_Bd}.
\end{proof}

\begin{lem}\label{LEM.JFC_Preservation}
Let $H_{0}$ and $H_{1}$ be Hilbert spaces. Let $T_{0},S_{0}\in\UBII\lc{}H_{0}\rc_{h}$, as well as $T_{1},S_{1}\in\UBII\lc{}H_{1}\rc_{h}$ commute strongly. If $\phi:W^{*}\lc{}T_{0}\rc\longrightarrow W^{*}\lc{}T_{1}\rc$ and $\bpsi:W^{*}\lc{}S_{0}\rc\longrightarrow W^{*}\lc{}S_{1}\rc$ are normal unital $^{*}$-homomorphisms  s.t.~$\phi\lc{}E_{T_{0}}\rc{}=E_{T_{1}}$ and $\bpsi\lc{}E_{S_{0}}\rc{}=E_{S_{1}}$, then 

\begin{itemize}
\item[1)] $\spec T_{1}\times S_{1}\subset\spec T_{0}\times S_{0}$ and $\mathcal{N}\lc{}E_{T_{0},S_{0}}\rc\subset\mathcal{N}\lc{}E_{T_{1},S_{1}}\rc$,

\item[2)] $\lc\phi\otimes\bpsi\rc\lc{}g\lc{}T_{0},S_{0}\rc\rc{}=g\lc{}T_{1},S_{1}\rc$ for all $g\in L^{\infty}\lc\spec T_{0}\times S_{0},dE_{T_{0},S_{0}}\rc$,

\item[3)] we have commutative diagram of normal unital surjective $^{*}$-homomorphisms

\begin{equation}\label{EQ.LEM.JFC_Preservation_1}
\begin{tikzcd}
L^{\infty}\lc\spec T_{0}\times S_{0},dE_{T_{0},S_{0}}\rc\arrow[rr,"\Gamma_{T_{0},S_{0}}"]\arrow[dd,"\res"] & & W^{*}\lc{}T_{0},S_{0}\rc\arrow[dd,"\phi\otimes\bpsi"] \\
& & \\
L^{\infty}\lc\spec T_{1}\times S_{1},dE_{T_{1},S_{1}}\rc\arrow[rr,"\Gamma_{T_{1},S_{1}}"] & & W^{*}\lc{}T_{1},S_{1}\rc{}
\end{tikzcd}
\end{equation}

\begin{reapply}
\end{reapply}

\noindent with $\res$ the restriction map given by $\spec T_{1}\times S_{1}\subset\spec T_{0}\times S_{0}$.
\end{itemize}
\end{lem}
\begin{proof}
We apply Lemma \ref{LEM.FC_Preservation_II} and Corollary \ref{COR.Wstar_TP} to $\phi$ and $\bpsi$. This constructs normal unital surjective $^{*}$-homomorphism $\phi\otimes\bpsi:W^{*}\lc{}T_{0},S_{0}\rc\longrightarrow W^{*}\lc{}T_{1},S_{1}\rc$. Note Equation \ref{EQ.SSEC.A_Fnd_CWstar_4} and Equation \ref{EQ.SSEC.A_Fnd_FC_8} show $\phi\otimes\bpsi$ is determined by $\lc\phi\otimes\bpsi\rc\lc{}g\lc{}T_{0}\rc{}h\lc{}S_{0}\rc\rc{}=g\lc{}T_{1}\rc{}h\lc{}S_{1}\rc$ for all $g\in L^{\infty}\lc\spec T_{0},dE_{T_{0}}\rc$ and $h\in L^{\infty}\lc\spec S_{0},dE_{S_{0}}\rc$. For all $Z,Z'\in\mathfrak{B}(\mathbb{R})$, construction of joint spectral measures shows

\begin{align*}
\lc\phi\otimes\bpsi\rc\lc{}E_{T_{0},S_{0}}(Z\times Z')\rc{} & = \lc\phi\otimes\bpsi\rc\lc{}E_{T_{0}}(Z)E_{S_{0}}(Z')\rc \phantom{\bigg)} \\
& = E_{T_{1}}(Z)E_{S_{1}}(Z')=E_{T_{1},S_{1}}(Z\times Z'). \phantom{\bigg)}
\end{align*}

\noindent Arguing as in the proof of Lemma \ref{LEM.FC_Preservation_II}, the above calculation implies $1)$. We see the restriction map $\res$ is well-defined. Using Proposition \ref{PRP.Wstar_Normal}, we directly verify $2)$ and $3)$ on elementary tensors. Using $\sigma$-strong continuity, we conclude by strong density.
\end{proof}


\section{Maps of unbounded operators}\label{SEC.A_Maps}

In Subsection \ref{SSEC.A_Maps_SR}, we discuss strong resolvent convergence and resolvent-preserving maps of unbounded operators. Strong resolvent convergence provides suitable notion of continuity. Given uniform neighbourhood of supports, evaluation of functional calculus on fixed bounded continuous functions is strong resolvent continuous. We extend to joint functional calculus. Resolvent-preserving maps are strong resolvent continuous and preserve functional calculus. Examples are twisting and compression maps.\par
In Subsection \ref{SSEC.A_Maps_Compression}, we introduce abstract and concrete compression maps. They are given by left-~and right-multiplication with projections. In the abstract case, we compress $W^{*}$-algebras. In the concrete case, we thus compress self-adjoint unbounded operators on a Hilbert space by reducing subspaces. We extend abstract compression maps to spaces of measurable operators in Subsection \ref{SSEC.B_JFC_L2Red}.


\subsection{Strong resolvent continuity and resolvent-preservation}\label{SSEC.A_Maps_SR}

We define strong resolvent convergence, prove strong resolvent continuity of functional calculus in Lemma \ref{LEM.FC_SR} and review sufficient conditions. We then give two standard approximations and discuss resolvent-preserving maps. Standard reference for strong resolvent convergence is \cite{BK.deOli.2009.OpAlg_Quantum_Dynamics}.


\subsubsection*{Strong resolvent convergence of self-adjoint unbounded operators}

Note Definition \ref{DFN.SR} gives strong resolvent convergence on Hilbert spaces. We use the latter to extend $2)$ in Lemma \ref{LEM.JFC_Preservation} to suitable unbounded functions in Corollary \ref{COR.JFC_Preservation}.\par
In Subsection \ref{SSEC.NCDS_NCD_Operators}, Definition \ref{DFN.SR_Subspace} generalises to strong resolvent convergence on Hilbert subspaces for use in the Kato-Robinson theorem \lc{}cf.~Theorem 10.4.2 in \cite{BK.deOli.2009.OpAlg_Quantum_Dynamics}\rc{}. In the appendix, we only use strong resolvent convergence on Hilbert spaces. 

\begin{dfn}\label{DFN.SR}
Let $H$ be a Hilbert space. We call $\lset{}T_{n}\rset_{n\in\mathbb{N}}\subset\UBII(H)_{h}$ strong resolvent convergent to $T\in\UBII(H)_{h}$ on $H$ if $R_{i}(T)=\s$-$\lim_{n\in\mathbb{N}}R_{i}(T_{n})$.
\end{dfn}

\begin{ntn}\label{NTN.SR}
Let $T=\sr$-$\lim_{n\in\mathbb{N}}T_{n}$ on $H$ denote strong resolvent convergence. We drop \textit{on H} if $H$ is clear from context. We extend to strong resolvent convergence on Hilbert subspaces in Notation \ref{NTN.SR_Subspace}.
\end{ntn}

\begin{rem}\label{REM.SR_Equivalence}
We equivalently use $R_{-i}$ in Definition \ref{DFN.SR} \lc{}cf.~Lemma 10.1.5 in \cite{BK.deOli.2009.OpAlg_Quantum_Dynamics}\rc{}. Moreover, note uniform boundedness and strong resolvent convergence together are equivalent to strong convergence \lc{}cf.~Proposition 10.1.13 in \cite{BK.deOli.2009.OpAlg_Quantum_Dynamics}\rc{}.
\end{rem}

Note Lemma \ref{LEM.FC_SR} is based on the case of one self-adjoint unbounded operator as per Remark \ref{REM.FC_SR}. We recover this one-variable case using the identity as second one.

\begin{rem}\label{REM.FC_SR}
Proposition 10.1.9 in \cite{BK.deOli.2009.OpAlg_Quantum_Dynamics} shows we have $T=\sr$-$\lim_{n\in\mathbb{N}}T_{n}$ if and only if $g(T)=\s$-$\lim_{n\in\mathbb{N}}g(T_{n})$ for all $g\in C_{b}(\mathbb{R})$. Lemma \ref{LEM.FC_SR} yields the first direction given two strongly commuting self-adjoint unbounded operators. We recover the one-variable case by setting $S=S_{n}=I$ for all $n\in\mathbb{N}$ and $g=g\cdot 1\in C_{b}(\mathbb{R})$ in Lemma \ref{LEM.FC_SR}.
\end{rem}

\begin{lem}\label{LEM.FC_SR}
Let $T=\sr$-$\lim_{n\in\mathbb{N}}T_{n}$ and $S=\sr$-$\lim_{n\in\mathbb{N}}S_{n}$ on $H$. Let $T$ and $S$ commute strongly. For all $n\in\mathbb{N}$, let $T_{n}$ and $S_{n}$ commute strongly. Set

\begin{align}\label{EQ.LEM.FC_SR_1}
X_{T,S}:=\overline{\bigcup_{n\in\mathbb{N}}\spec T_{n}\times\spec S_{n}}\subset\mathbb{R}\times\mathbb{R}.
\end{align}

\noindent If $g\in C_{b}\lc{}X_{T,S}\rc$, then $g\lc{}T,S\rc{}=\s$-$\lim_{n\in\mathbb{N}}g\lc{}T_{n},S_{n}\rc$.
\end{lem}

\begin{proof}
As $X_{T,S}$ is closed by hypothesis and contains all spectral measure supports in use \lc{}cf.~Corollary 10.2.2 in \cite{BK.deOli.2009.OpAlg_Quantum_Dynamics}\rc{}, we assume $g\in C_{b}\lc\mathbb{R}\times\mathbb{R}\rc$ without loss of generality. For all $g_{0},g_{1}\in C_{0}(\mathbb{R})$, sequential strong continuity of multiplication yields

\begin{align*}
\big(g_{0}\otimes g_{1}\big)\lc{}T,S\rc{} & = g_{0}(T)g_{1}(S) \phantom{\bigg)} \\
& = \s\textrm{-}\lim_{n\in\mathbb{N}}\hspace{0.025cm} g_{0}(T_{n})g_{1}(S_{n}) \phantom{\bigg)} \\
& = \s\textrm{-}\lim_{n\in\mathbb{N}}\hspace{0.025cm} \big(g_{0}\otimes g_{1}\big)\lc{}T_{n},S_{n}\rc \phantom{\bigg)}
\end{align*}

\noindent using the one-variable case as per Remark \ref{REM.FC_SR}. Thus approximating uniformly in norm shows our claim for all $g\in C_{0}\lc\mathbb{R}\times\mathbb{R}\rc$. If $g\in C_{b}\lc\mathbb{R}\times\mathbb{R}\rc$, then fix a monotone sequence of mollifying functions and argue as in the proof of Proposition 10.1.9 in \cite{BK.deOli.2009.OpAlg_Quantum_Dynamics}.
\end{proof}

\begin{cor}\label{COR.JFC_Preservation}
Assume the setting of Lemma \ref{LEM.JFC_Preservation}. For all real $g\in\SII\lc{}E_{T_{0},S_{0}}\rc$ s.t~

\begin{itemize}
\item[1)] $(t,s)\mapsto g_{\varepsilon}(t,s):=g\lc{}t+\varepsilon,\s+\varepsilon\rc$ lies in $C_{b}\lc\spec T_{0}\times S_{0}\rc$ for all $\varepsilon>0$,

\item[2)] $g\lc{}T_{1},S_{1}\rc{}=\sr$-$\lim_{\varepsilon\downarrow 0}g_{\varepsilon}\lc{}T_{1},S_{1}\rc$ on $H_{1}$,
\end{itemize}

\noindent we have $g\in\SII\lc{}E_{T_{1},S_{1}}\rc$ with $g\lc{}T_{1},S_{1}\rc{}=\sr$-$\lim_{\varepsilon\downarrow 0}\hspace{0.025cm}\lc\phi\otimes\bpsi\rc\lc{}g_{\varepsilon}\lc{}T_{0},S_{0}\rc\rc$ on $H_{1}$.
\end{cor}
\begin{proof}
We know $g\in\SII\lc{}E_{T_{1},S_{1}}\rc$ by $1)$ in Lemma \ref{LEM.JFC_Preservation}. For all $\varepsilon>0$, we apply $2)$ in Lemma \ref{LEM.JFC_Preservation} to $g_{\varepsilon}\in C_{b}\lc\spec T_{0}\times S_{0}\rc$. Thus $\lc\phi\otimes\bpsi\rc\lc{}g_{\varepsilon}\lc{}T_{0},S_{0}\rc\rc{}=g_{\varepsilon}\lc{}T_{1},S_{1}\rc$, hence we conclude by $2)$ and $\varepsilon\downarrow 0$.
\end{proof}

\begin{rem}\label{REM.JFC_Preservation}
In the sense of Corollary \ref{COR.JFC_Preservation}, Lemma \ref{LEM.JFC_Preservation} gives conditions to pull back unbounded joint functional calculus. Lemma \ref{LEM.JFC_Preservation} and Lemma \ref{LEM.FC_SR} further show Corollary \ref{COR.JFC_Preservation} applies to all $g\in C_{b}(X)$ for compact $X\subset\mathbb{R}\times\mathbb{R}$ with $\delta>0$ s.t.~

\begin{align}\label{EQ.REM.JFC_Preservation_1}
\bigcup_{0<\varepsilon<\delta}\spec T_{1}+\varepsilon I\times\spec S_{1}+\varepsilon I\subset X.
\end{align}

\noindent Note $X_{T,S}\subset X$ as per Equation \ref{EQ.LEM.FC_SR_1} in this case.
\end{rem}

\begin{prp}\label{PRP.SR}
We have $T=\sr$-$\lim_{n\in\mathbb{N}}T_{n}$ on $H$ if there exists

\begin{itemize}
\item[1)] $a\in\rsl T$ s.t.~$a\in\bigcap_{n\in\mathbb{N}}\rsl T_{n}$ and $R_{a}(T)=\s$-$\lim_{n\in\mathbb{N}}R_{a}(T_{n})$,

\item[2)] or core $\HII$ of $T$ s.t.~$\HII\subset\bigcap_{n\in\mathbb{N}}\dom T_{n}$ and $T(u)=\|.\|_{H}$-$\lim_{n\in\mathbb{N}}T_{n}(u)$ for all $u\in\HII$.
\end{itemize}
\end{prp}
\begin{proof}
Get $1)$ by Proposition 10.1.23, resp.~$2)$ by Proposition 10.1.18 in \cite{BK.deOli.2009.OpAlg_Quantum_Dynamics}.
\end{proof}


\subsubsection*{Two standard approximations}

Using cut-off sequences, functional calculus lets us approximate self-adjoint and positive unbounded operators on Hilbert spaces in strong resolvent convergence. Let $H$ be a Hilbert space.

\begin{lem}\label{LEM.SR_Approximation}
For all $T\in\UBII(H)_{+}$, we have

\begin{itemize}
\item[1)] $T=\sr$-$\lim_{n\in\mathbb{N}}\min\hspace{-0.033875cm} \lset{}T,n\rset$,

\item[2)] $\lambda\notin\spec T\ \textrm{if and only if}\ \lambda\notin\spec\min\hspace{-0.033875cm} \lset{}T,n\rset\ \textrm{for a.e.}\ n\in\mathbb{N}$.
\end{itemize}
\end{lem}
\begin{proof}
For all $n\in\mathbb{N}$, set $T_{n}:=\min\hspace{-0.033875cm} \lset{}T,n\rset$ and note $T_{n+1}\geq T_{n}$. For all $u\in H$, we have

\begin{align*}
\sup_{n\in\mathbb{N}}\hspace{0.025cm} \lgl T_{n}(u),u\rgl_{H}=\sup_{n\in\mathbb{N}}\hspace{0.025cm} \int_{\spec T}\min\{t,n\}dE_{T}^{u}=\begin{cases}\lgl T(u),u\rgl_{H} & \If\ u\in\dom\sqrt{T}, \\
\infty & \Else.
\end{cases}
\end{align*}

\noindent Thus $T_{n}\uparrow T$ monotonically in the sense of closed positive unbounded quadratic forms on $H$, hence get $1)$ by the Kato-Robinson theorem \lc{}cf.~Theorem 10.4.2 in \cite{BK.deOli.2009.OpAlg_Quantum_Dynamics}\rc{}.\par
We show $2)$. We know $\spec T=\bigcup_{n\in\mathbb{N}}\lset\lambda\leq n\ \vert\ \lambda\in\spec T\rset$. For all $n\in\mathbb{N}$, the spectral mapping theorem \lc{}cf.~Proposition 5.25 in \cite{BK.Sch.2012.Unbounded_Operators}\rc{} implies

\begin{align*}
\spec T_{n}=
\begin{cases} 
\spec T & \If\ \| T\|_{\BII(H)}\leq n, \\
\big\{\hspace{0.025cm} \lambda\leq n\ \vert\ \lambda\in\spec T\hspace{0.025cm} \big\}\cup\{n\} & \Else.
\end{cases}
\end{align*}

\noindent Let $\lambda\geq 0$. If $\lambda\notin\spec T$, then $\lambda\notin\lset\lambda\leq n\ \vert\ \lambda\in\spec T\rset{}=\spec T_{n}\setminus\{n\}$ for all $n\in\mathbb{N}$. If further $\lambda\in\spec T_{n_{0}}$ for $n_{0}\in\mathbb{N}$, then $\lambda=n_{0}$ and $n_{0}\notin\spec T$. We see $\lambda\notin\spec T$ implies $\lambda\notin\spec T_{n}$ for a.e.~$n\in\mathbb{N}$. We know $\lset\lambda\leq n\ \vert\ \lambda\in\spec T\rset\subset\PII\lc\spec T\rc$ is a monotonically increasing sequence. Thus $\lambda\notin\spec T_{n}$ for a.e.~$n\in\mathbb{N}$ implies $\lambda\notin\spec T$, hence $2)$ follows.
\end{proof}

\begin{cor}\label{COR.SR_Approximation}
For all $T\in\UBII(H)_{h}$, $T=\sr$-$\lim_{n\in\mathbb{N}}E_{T}\lc{}[-n,n]\rc{}T$.
\end{cor}
\begin{proof}
Set $T_{+}:=E_{T}\lc{}[0,\infty)\rc{}T$ and $T_{-}:=-E_{T}\lc\lc{}-\infty,0\rb\rc{}T$. Get $T=T_{+}-T_{-}$ by functional calculus. Using Proposition \ref{PRP.JFC_Strong_Com}, we know $T_{+},T_{-}\in\UBII(H)_{+}$ commute strongly since $R_{i}\lc{}T_{+}\rc{},R_{i}\lc{}T_{-}\rc\in W^{*}(T)$ commute. For all $n\in\mathbb{N}$, functional calculus implies
 
\begin{align}\label{EQ.COR.SR_Approximation_1}
E_{T}\lc{}[-n,n]\rc{}T=E_{T_{+}}\lc[0,n]\rc{}T_{+}-E_{T_{-}}\lc[0,n]\rc{}T_{-}.
\end{align}

\noindent Summands in Equation \ref{EQ.COR.SR_Approximation_1} commute strongly. Note $(t,s)\mapsto g(t,s):=R_{i}\lc{}t-s\rc$ lies in $C_{b}\lc\mathbb{R}\times\mathbb{R}\rc$. If $S=\sr$-$\lim_{n\in\mathbb{N}}E_{S}\lc[0,n]\rc{}S$ for all $S\geq 0$ on $H$, then Lemma \ref{LEM.FC_SR} shows

\begin{align*}
R_{i}(T) = R_{i}\lc{}T_{+}-T_{-}\rc{} & = g\lc{}T_{+},T_{-}\rc \phantom{\bigg)} \\
& = \s\textrm{-}\lim_{n\in\mathbb{N}}\hspace{0.025cm} g\lc{}E_{T_{+}}\lc[0,n]\rc{}T_{+},E_{T_{-}}\lc[0,n]\rc{}T_{-}\rc \phantom{\bigg)} \\
& = \s\textrm{-}\lim_{n\in\mathbb{N}}\hspace{0.025cm} R_{i}\lc{}E_{T_{+}}\lc[0,n]\rc{}T_{+}-E_{T_{-}}\lc[0,n]\rc{}T_{-}\rc \phantom{\bigg)} \\
& = \s\textrm{-}\lim_{n\in\mathbb{N}}\hspace{0.025cm} R_{i}\lc{}E_{T}\lc{}[-n,n]\rc{}T\rc{}. \phantom{\bigg)}
\end{align*}
 
If $S=\sr$-$\lim_{n\in\mathbb{N}}E_{S}\lc[0,n]\rc{}S$ for all $S\geq 0$ on $H$, then the above calculation shows our claim follows. Let $S\in\UBII(H)_{+}$. For all $n\in\mathbb{N}$, functional calculus implies

\begin{align}\label{EQ.COR.SR_Approximation_2}
E_{S}\lc[0,n]\rc{}S=\min\hspace{-0.04375cm} \vstretch{1.125}{\{}\hspace{0.0075cm} S,n\hspace{0.0075cm} \vstretch{1.125}{\}}-n\cdot \big(I-E_{S}\lc[0,n]\rc\big).
\end{align}

\noindent Summands in Equation \ref{EQ.COR.SR_Approximation_2} commute strongly. Note Lemma \ref{LEM.SR_Approximation} ensures we have $S=\sr$-$\lim_{n\in\mathbb{N}}\min\hspace{-0.033875cm} \lset{}S,n\rset$. We moreover have pointwise convergence $\lim_{n\in\mathbb{N}}n\lc{}1-\chi_{[0,n]}\rc{}=0$ on $[0,\infty)$, i.e.~uniformly bounded pointwise limit $\lim_{n\in\mathbb{N}}R_{i}\lc{}n\lc{}1-\chi_{[0,n]}\rc\rc{}=R_{i}(0)$ in $C_{0}\lc{}[0,\infty)\rc$. Thus $\sr$-$\lim_{n\in\mathbb{N}}n\lc{}I-E_{S}\lc[0,n]\rc\rc{}=0$. Using Lemma \ref{LEM.FC_SR} as above for each summand on the right-hand side of Equation \ref{EQ.COR.SR_Approximation_2} a separate variable, we obtain our claim.
\end{proof}


\subsubsection*{Resolvent-preserving maps}

Lemma \ref{LEM.RP} shows resolvent-preserving maps are strong resolvent continuous and preserve functional calculus. Both twisting and compression maps are resolvent-preserving. Let $H_{0}$ and $H_{1}$ be Hilbert spaces.

\begin{dfn}\label{DFN.RP}
Let $\phi:\UBII\lc{}H_{0}\rc\longrightarrow\UBII\lc{}H_{1}\rc$ be a linear map s.t.~$\phi\lc\BII\lc{}H_{0}\rc\rc\subset\BII\lc{}H_{1}\rc$ and $\DII\subset\UBII\lc{}H_{0}\rc_{h}$. We say that $\phi$ is resolvent-preserving using $\DII(\phi):=\DII$ if

\begin{itemize}
\item[1)] $\phi:\BII\lc{}H_{0}\rc\longrightarrow\BII\lc{}H_{1}\rc$ is bounded and normal,

\item[2)] $\phi\big(R_{\pm i}(T)\big)=R_{\pm i}\lc\phi(T)\rc$ for all $T\in\DII(\phi)$.
\end{itemize}
\end{dfn}

\begin{lem}\label{LEM.RP}
Let $\phi:\UBII\lc{}H_{0}\rc\longrightarrow\UBII\lc{}H_{1}\rc$ be a resolvent-preserving map.

\begin{itemize}
\item[1)] Let $T\in\DII(\phi)$. If $T=\sr$-$\lim_{n\in\mathbb{N}}T_{n}$ on $H_{0}$ and $\lset{}T_{n}\rset_{n\in\mathbb{N}}\subset\DII(\phi)$, then

\begin{align}\label{EQ.LEM.RP_1}
\phi(T)=\sr\textrm{-}\lim_{n\in\mathbb{N}}\hspace{0.025cm} \phi(T_{n})\ \textrm{on}\ H_{1}.    
\end{align}

\begin{reapply}
\end{reapply}

\item[2)] For all $T\in\DII(\phi)$, we have

\begin{itemize}
\item[2.1)] $\phi:W^{*}(T)\longrightarrow W^{*}\lc\phi(T)\rc$ is a normal unital $^{*}$-isomorphism,

\item[2.2)] $\phi\lc{}E_{T}\rc{}=E_{\phi(T)}$ and Lemma \ref{LEM.FC_Preservation_II} applies to $\phi:W^{*}(T)\longrightarrow W^{*}\lc\phi(T)\rc$.
\end{itemize}

\begin{reapply}
\end{reapply}

\end{itemize}
\end{lem}
\begin{proof}
Let $T\in\DII(\phi)$. If $T=\sr$-$\lim_{n\in\mathbb{N}}T_{n}$ on $H_{0}$ and $\lset{}T_{n}\rset_{n\in\mathbb{N}}\subset\DII(\phi)$, then we calculate

\begin{align*}
R_{i}\lc\phi(T)\rc{} = \phi\big(R_{i}(T)\big) & = \phi\big(\s\textrm{-}\lim_{n\in\mathbb{N}}\hspace{0.025cm} R_{i}(T_{n})\big) \phantom{\bigg)} \\
& = \s\textrm{-}\lim_{n\in\mathbb{N}}\hspace{0.025cm} \phi\big(R_{i}(T_{n})\big) = \s\textrm{-}\lim_{n\in\mathbb{N}}\hspace{0.025cm} R_{i}\lc\phi(T_{n})\rc{}. \phantom{\bigg)}
\end{align*}

\noindent This shows $1)$. We show $2)$. Note $\phi:C^{*}\lc{}R_{\pm i}(T)\rc\longrightarrow C^{*}\lc{}R_{\pm i}\lc\phi(T)\rc\rc$ maps $C^{*}$-generators onto by hypothesis. Ergo $\phi\vert_{C^{*}\lc{}R_{\pm i}(T)\rc{}}$ is unital $^{*}$-isomorphism s.t.~$\phi\lc{}g(T)\rc{}=g\lc\phi(T)\rc$ for all $g\in C_{c}(\mathbb{R})$. We know $\phi\vert_{\BII\lc{}H_{0}\rc{}}$ is normal. Therefore, $\sigma$-weak closure of $\phi\vert_{C^{*}\lc{}R_{\pm i}(T)\rc{}}$ exists and is normal unital $^{*}$-isomorphism $\phi:W^{*}(T)\longrightarrow W^{*}\lc\phi(T)\rc$. Thus Lemma \ref{LEM.FC_Preservation_I}, hence Lemma \ref{LEM.FC_Preservation_II} applies as claimed. Get $2)$.
\end{proof}

\begin{cor}\label{COR.Unbd_Twist_RP}
Let $\phi:H_{0}\longrightarrow H_{1}$ be a linear or anti-linear isometric isomorphism.

\begin{itemize}
\item[1)] $\phi^{\dagger}$ is resolvent-preserving using $\DII\lc\phi^{\dagger}\rc{}=\UBII\lc{}H_{0}\rc_{h}$.

\item[2)] If $T_{0}\in\UBII\lc{}H_{0}\rc_{h}$ and $T_{1}\in\UBII\lc{}H_{1}\rc_{h}$ s.t.~$\phi^{\dagger}\lc{}E_{T_{0}}\rc{}=E_{T_{1}}$, then $\phi^{\dagger}\lc{}T_{0}\rc{}=T_{1}$.
\end{itemize}
\end{cor}
\begin{proof}
Using Proposition \ref{PRP.Unbd_Twist}, we directly verify $1)$. Thus $E_{\phi^{\dagger}\lc{}T_{0}\rc{}}=\phi^{\dagger}\lc{}E_{T_{0}}\rc{}=E_{T_{1}}$ by $2)$ in Lemma \ref{LEM.RP} and hypothesis, hence $2)$ follows by the spectral theorem.
\end{proof}


\subsection{Compression maps, reducing subspaces and spectral gaps}\label{SSEC.A_Maps_Compression}

We introduce abstract and concrete compression maps. Reducing subspaces are used to define subsets for which compression maps are resolvent-preserving. We then apply compression to get useful standard results concerning spectral gaps. Standard reference for reducing subspaces is \cite{BK.deOli.2009.OpAlg_Quantum_Dynamics}.


\subsubsection*{Compression maps}

Definition \ref{DFN.Compression_Abstract_Bd} gives abstract compression maps as per Remark \ref{REM.Abstract_Concrete}, and Definition \ref{DFN.Compression_Concrete} gives concrete ones. Following our discussion in Subsection \ref{SSEC.B_JFC_L2Red}, Definition \ref{DFN.Compression_Abstract} extends Definition \ref{DFN.Compression_Abstract_Bd} to spaces of measurable operators. We unify in Corollary \ref{COR.Wstar_L2Red_III} for spaces of measurable operators.

\begin{rem}\label{REM.Abstract_Concrete}
Following \cite{BK.Tak.1979.OpAlg_I}\cite{BK.Tak.2003.OpAlg_II}\cite{BK.Tak.2003.OpAlg_III}, \textit{abstract} signifies that an object or property is independent of representation whereas \textit{concrete} assumes representation.
\end{rem}

Let $M$ be a $W^{*}$-algebra and $p\in M$ a projection. If $A\subset M$ is a $C^{*}$-subalgebra, then $pC^{*}\lc{}A,p\rc{}p\subset M$ is one. If $N\subset M$ is a $W^{*}$-subalgebra, then $pC^{*}\lc{}N,p\rc{}p\subset M$ is one.

\begin{dfn}\label{DFN.Compression_Abstract_Bd}
Let $M$ be a $W^{*}$-algebra. We consider $C^{*}$-subalgebra $A\subset M$. For all projections $p\in M$, we define 

\begin{itemize}
\item[1)] orthogonal projection $p^{\perp}:=1_{M}-p\in M$,

\item[2)] compressed $C^{*}$-subalgebra $A[p]:=pC^{*}\lc{}A,p\rc{}p\subset M$,

\item[3)] the compression map $\comp:A[1_{M}]\longrightarrow A[p]$ by setting

\begin{align}\label{EQ.DFN.Compression_Abstract_Bd_1}
\comp x:=pxp
\end{align}

\begin{reapply}
\end{reapply}

\noindent for all $x\in A[1_{M}]$.
\end{itemize}
\end{dfn}

\begin{rem}\label{REM.Compression_Abstract_Bd}
If $p\in A$, then $A[p]=pAp$. If $p=1_{M}$, then we recover unitalisation.
\end{rem}

\begin{prp}\label{PRP.Compression_Abstract_Bd}
Let $M$ be a $W^{*}$-algebra and $N\subset M$ a $W^{*}$-subalgebra. If $p\in M$ is a projection, then $\comp:N[1_{M}]\longrightarrow N[p]$ is a completely positive, normal, unital and surjective bounded linear map.
\end{prp}
\begin{proof}
Complete positivity is given in Example \ref{BSP.Wstar_CP_III}. Bounded weak continuity and Proposition \ref{PRP.Wstar_Normal} imply normality. All remaining claims follow by construction.
\end{proof}


\pagebreak


Let $N\subset M$ be a $W^{*}$-subalgebra. Proposition \ref{PRP.Wstar_Unitalisation} states $N[1_{M}]=N\oplus\langle 1_{M}^{\perp}\rangle_{\mathbb{C}}$. We directly verify $N1_{N}^{\perp}=1_{N}^{\perp}N=0$. We have $N[1_{M}][1_{N}]=N$ and commutative diagram

\smallskip

\begin{equation}\label{EQ.SSEC.A_Maps_Compression_1}
\begin{tikzcd}
N\arrow[rr,hook]\arrow[rrrr, bend right=43, "\id_{N}"] & & N\oplus\langle 1_{M}^{\perp}\rangle_{\mathbb{C}}\arrow[rr,"\comunit"] & & N
\end{tikzcd}   
\end{equation}

\medskip

\noindent of normal $^{*}$-homomorphisms. We extend Diagram \ref{EQ.SSEC.A_Maps_Compression_1} in Subsection \ref{SSEC.B_JFC_L2Red}. Following this, Corollary \ref{COR.Wstar_Compression_Preservation_II} shows choice of unit only involves values at zero.\par
Let $H$ be a Hilbert space. If $V\subset H$ is a Hilbert subspace and $\pi_{V}:H\longrightarrow V$ its Hilbert space projection, we use positivity-preserving canonical inclusion $\UBII(V)\subset\UBII(H)$ by setting

\begin{align}\label{EQ.SSEC.A_Maps_Compression_2}
T=\pi_{V}T\pi_{V}
\end{align}

\noindent for all $T\in\UBII(V)$. For details on inclusions and partial order for spaces of unbounded operators, we refer to Subsection \ref{SSEC.A_Fnd_Unbd}, in particular Remark \ref{REM.Unbd_PO}.

\begin{dfn}\label{DFN.Compression_Concrete}
Let $H$ be a Hilbert space. For all Hilbert subspaces $V\subset H$, i.e.~for $\|.\|_{H}$-closed ones, let $\pi_{V}:H\longrightarrow V$ denote its Hilbert space projection and we define

\begin{itemize}
\item[1)] orthogonal projection $\pi_{V}^{\perp}:=I_{H}-\pi_{V}\in\BII(H)$,

\item[2)] inclusion $\UBII(V)\subset\UBII(H)$ as per Equation \ref{EQ.SSEC.A_Maps_Compression_2},

\item[3)] the compression map $\comV:\UBII(H)\longrightarrow\UBII(V)$ by setting

\begin{align}\label{EQ.DFN.Compression_Concrete_1}
\comV T:=\pi_{V}T\pi_{V}  
\end{align}

\begin{reapply}
\end{reapply}

\noindent for all $T\in\UBII(H)$.
\end{itemize}
\end{dfn}

\begin{prp}\label{PRP.Compression_Concrete}
Let $H$ be a Hilbert space. If $V\subset H$ is a Hilbert subspace, then $\BII(H)[\pi_{V}]=\BII(V)$ and $\comV:\BII(H)\longrightarrow \BII(V)$ is a completely positive, normal, unital and surjective bounded linear map.
\end{prp}
\begin{proof}
Apply Proposition \ref{PRP.Compression_Abstract_Bd} for $M=N=\BII(H)$ and $p=\pi_{V}$.
\end{proof}


\subsubsection*{Reducing subspaces}

Proposition \ref{PRP.Compression_Concrete} shows compression maps satisfy $1)$ in Definition \ref{DFN.RP}. Reducing subspaces, resp.~reducible operators, yield the definition domains for concrete compression maps. Note Equation \ref{EQ.DFN.Reducible_1} below reduces to the obvious commutation relation in the bounded case.\par


\pagebreak


Let $H$ be a Hilbert space.

\begin{dfn}\label{DFN.Reducible}
Let $V\subset H$ be a Hilbert subspace.

\begin{itemize}
\item[1)] We say that $T\in\UBII(H)_{h}$ is $V$-reducible and call $V$ a reducing subspace of $T$ if

\begin{align}\label{EQ.DFN.Reducible_1}
\pi_{V}T\subset T\pi_{V}.    
\end{align}

\begin{reapply}
\end{reapply}

\item[2)] Let $\UBII_{V}(H)$ be the set of all $V$-reducible $T\in\UBII(H)_{h}$. For all $T\in\UBII(H)$, set

\begin{align}\label{EQ.DFN.Reducible_2}
\restr{0.925}{T}{V}:=\comV T.    
\end{align}

\begin{reapply}
\end{reapply}

\end{itemize}
\end{dfn}

\begin{rem}
Note $\UBII(V)\subset\UBII_{V}(H)$. For all $T\in\UBII(V)$, get $\restr{0.925}{T}{V}=\comV T=T$.
\end{rem}

\begin{ntn}\label{NTN.Reducible}
Let $V\subset H$ be a Hilbert subspace. For all $T\in\UBII(H)$, we write $\restr{0.925}{T}{V}$ if we consider $\comV T$ as operator on $V$.
\end{ntn}

Let $V\subset H$ be a Hilbert subspace. Lemma 9.8.4 in \cite{BK.deOli.2009.OpAlg_Quantum_Dynamics} shows we have $T\in\UBII_{V}(H)$ if and only if $\pi_{V}\lc\dom T\rc\subset\dom T$ and $T\pi_{V}\lc\dom T\rc\subset V$. Since $\UBII_{V}(H)=\UBII_{V^{\perp}}(H)$, we may replace $V$ with $V^{\perp}$ in all statements concerning reducing subspaces.

\begin{bsp}\label{BSP.Reducible}
If $T\in\UBII(H)_{h}$, then $T$ is reduced by $\overline{\im T}:=\overline{\im T}^{\|.\|_{H}}$ and $\ker T$.
\end{bsp}

\begin{prp}\label{PRP.Reducible}
Let $V\subset H$ be a Hilbert subspace.

\begin{itemize}
\item[1)] For all $T\in\UBII_{V}(H)$, we have

\begin{itemize}
\item[1.1)] $\restr{0.925}{T}{V}\in\UBII\lc{}V\rc_{h}$ and $T\pi_{V}=\comV T$,

\item[1.2)] $\restr{0.925}{T}{V}\in\UBII\lc{}V\rc_{+}$ if $T\in\UBII(H)_{+}$,

\item[1.3)] $T=\comV T+\comVperp T$,
\end{itemize}

\begin{reapply}
\end{reapply}

\item[2)] For all Hilbert subspaces $W\subset V$, we have 

\begin{itemize}
\item[2.2)] $\comW=\comW\circ\comV$,

\item[2.3)] $\UBII_{W}(H)\subset\UBII_{V}(H)$.
\end{itemize}
\end{itemize}
\end{prp}
\begin{proof}
Theorem 9.8.3 in \cite{BK.deOli.2009.OpAlg_Quantum_Dynamics} implies $1)$ at once. We directly verify $2)$. 
\end{proof}

\begin{prp}\label{PRP.Compression_Concrete_RP}
If $V\subset H$ is a Hilbert subspace, then $\comV$ is resolvent-preserving using $\DII\lc\comV\rc{}=\UBII_{V}(H)$.
\end{prp}
\begin{proof}
We directly verify $\comV:\UBII(H)\longrightarrow\UBII(V)$ is linear. Proposition \ref{PRP.Compression_Concrete} and $1.1)$ in Proposition \ref{PRP.Reducible} immediately imply all conditions except $2)$ in Definition \ref{DFN.RP} are satisfied. We show the latter.\par


\pagebreak


Let $T\in\UBII_{V}(H)$. Since $\restr{0.925}{T}{V}\in\UBII\lc{}V\rc_{h}$ and thereby $\pm i\in\rsl(\restr{0.925}{T}{V})$, get

\begin{align}\label{EQ.PRP.Compression_Concrete_RP_1}
\im \restr{0.925}{T}{V}\mp iI_{V}=\dom R_{\pm i}(\restr{0.925}{T}{V})=V.
\end{align}

\noindent For all $v\in V$, let $u_{v}\in\im \restr{0.925}{T}{V}\mp iI_{V}\subset V$ s.t.~$v=\lc\restr{0.925}{T}{V}\mp iI_{V}\rc\lc{}u_{v}\rc$. We calculate

\begin{align}\label{EQ.PRP.Compression_Concrete_RP_2}
\comV R_{\pm i}(T)(v)=\pi_{V}\bigg(R_{\pm i}(T)\big(\pi_{V}\big(T\lc{}u_{v}\rc\mp iu_{v}\big)\big)\bigg)=\pi_{V}\bigg(R_{\pm i}(T)\big(T\lc{}u_{v}\rc\mp iu_{v}\big)\bigg)=u_{v}.
\end{align}

\noindent Injectivity of $\restr{0.925}{T}{V}\mp iI_{V}$ ensures Equation \ref{EQ.PRP.Compression_Concrete_RP_2} implies

\begin{align}\label{EQ.PRP.Compression_Concrete_RP_3}
\comV R_{\pm i}(T)=R_{\pm i}(\restr{0.925}{T}{V}).
\end{align}

\noindent Under canonical inclusion $\UBII(V)\subset\UBII(H)$ mapping $S\mapsto\pi_{V}S\pi_{V}$, note $\restr{0.925}{T}{V}=\comV T$ by definition. Equation \ref{EQ.PRP.Compression_Concrete_RP_3} therefore shows

\begin{align}\label{EQ.PRP.Compression_Concrete_RP_4}
\comV R_{\pm i}(T)=\comV R_{\pm i}\lc\comV T\rc{}=\pi_{V}R_{\pm i}\lc\comV T\rc\pi_{V}.
\end{align}

\noindent This is $2)$ in Definition \ref{DFN.RP}.
\end{proof}

\begin{lem}\label{LEM.Compression_Preservation_I}
If $T\in\UBII_{V}(H)$, then $\restr{0.925}{T}{V}\in\UBII\lc{}V\rc_{h}$ and we have

\begin{itemize}
\item[1)] $\spec \restr{0.925}{T}{V}\subset\spec T$ and $\mathcal{N}\lc{}E_{\restr{0.925}{T}{V}}\rc\subset\mathcal{N}\lc{}E_{T}\rc$,

\item[2)] normal unital surjective $^{*}$-homomorphism $\comV:W^{*}(T)\longrightarrow W^{*}(\restr{0.925}{T}{V})$ s.t.~

\begin{align}\label{EQ.LEM.Compression_Preservation_I_1}
\comV g(T)=\restr{0.925}{g(T)}{V}=g(\restr{0.925}{T}{V})=\comV g\lc\comV T\rc{}
\end{align}

\begin{reapply}
\end{reapply}

\noindent for all $g\in L^{\infty}\lc\spec T,dE_{T}\rc$,

\item[3)] commutative diagram of normal unital surjective $^{*}$-homomorphisms

\begin{equation}\label{EQ.LEM.Compression_Preservation_I_2}
\begin{tikzcd}
L^{\infty}\lc\spec T,dE_{T}\rc\arrow[rr,"\Gamma_{T}"]\arrow[dd,"\res"] & & W^{*}\lc{}T_{0}\rc\arrow[dd,"\comV"] \\
& & \\
L^{\infty}\lc\spec \restr{0.925}{T}{V},dE_{\restr{0.925}{T}{V}}\rc\arrow[rr,"\Gamma_{\restr{0.925}{T}{V}}"] & & W^{*}(\restr{0.925}{T}{V})
\end{tikzcd}
\end{equation}

\begin{reapply}
\end{reapply}

\noindent with $\res$ the restriction map given by $\spec \restr{0.925}{T}{V}\subset\spec T$.
\end{itemize}
\end{lem}
\begin{proof}
Proposition \ref{PRP.Compression_Concrete_RP} shows Lemma \ref{LEM.RP} applies. Thus Lemma \ref{LEM.FC_Preservation_II} applies to $\phi=\comV$ since $\DII\lc\comV\rc{}=\UBII_{V}(H)$ by hypothesis, hence our claims follow.
\end{proof}

\begin{cor}\label{COR.Compression_Preservation_I}
Let $T\in\UBII_{V}(H)$.

\begin{itemize}
\item[1)] $W^{*}(T)\subset\UBII_{V}(H)$.

\item[2)] For all $g\in L^{\infty}\lc\spec T,dE_{T}\rc$, $g(T)=g(\restr{0.925}{T}{V})\oplus g(\restr{0.925}{T}{V^{\perp}})\in\BII(V)\oplus\BII(V^{\perp})\subset\BII(H)$.

\item[3)] If $T=\sr$-$\lim_{n\in\mathbb{N}}T_{n}$ on $H$ and $\lset{}T_{n}\rset_{n\in\mathbb{N}}\subset\UBII_{V}(H)$, then $\restr{0.925}{T}{V}=\sr$-$\lim_{n\in\mathbb{N}}\restr{0.925}{T_{n}}{V}$ on $V$.
\end{itemize}
\end{cor}
\begin{proof}
Get $1)$ and $2)$ by Proposition \ref{PRP.Reducible} and Lemma \ref{LEM.Compression_Preservation_I}. Proposition \ref{PRP.Compression_Concrete_RP} shows $3)$ is $1)$ in Lemma \ref{LEM.RP} applied to $\phi=\comV$ using $\DII\lc\comV\rc{}=\UBII_{V}(H)$.
\end{proof}

\begin{cor}\label{COR.Compression_Preservation_II}
Let $T\in\UBII(H)_{h}$. We have $T\in\UBII_{V}(H)$ if and only if $\lb{}E_{T}(Z),\pi_{V}\rb{}=0$ for all $Z\in\mathfrak{B}(\mathbb{R})$. If $T\in\UBII_{V}(H)$, then $\lb{}g(T),\pi_{V}\rb{}=0$ for all $g\in L^{\infty}\lc\spec T,dE_{T}\rc$.
\end{cor}
\begin{proof}
If $T\in\UBII_{V}(H)$, then $\lb{}E_{T}(Z),\pi_{V}\rb$ for all $Z\in\mathfrak{B}(\mathbb{R})$ by $2)$ in Corollary \ref{COR.Compression_Preservation_I}. The converse is Lemma 9.8.6 in \cite{BK.deOli.2009.OpAlg_Quantum_Dynamics}. Apply $2)$ in Corollary \ref{COR.Compression_Preservation_I} for our final claim.
\end{proof}

\begin{lem}\label{LEM.Compression_Preservation_II}
If $T,S\in\UBII_{V}(H)$ commute strongly, then $\restr{0.925}{T}{V},\restr{0.925}{S}{V}\in\UBII\lc{}V\rc_{h}$ commute strongly and we have

\begin{itemize}
\item[1)] $\spec \restr{0.925}{T}{V}\times \restr{0.925}{S}{V}\subset\spec T\times S$ and $\mathcal{N}\lc{}E_{T,S}\rc\subset\mathcal{N}\lc{}E_{\restr{0.925}{T}{V},\restr{0.925}{S}{V}}\rc$,

\item[2)] normal unital surjective $^{*}$-homomorphism $\comV:W^{*}\lc{}T,S\rc\longrightarrow W^{*}\lc\restr{0.925}{T}{V},\restr{0.925}{S}{V}\rc$ s.t.~

\begin{align}\label{EQ.LEM.Compression_Preservation_II_1}
\comV g\lc{}T,S\rc{}=\restr{0.925}{g\lc{}T,S\rc{}}{V}=g\lc\restr{0.925}{T}{V},\restr{0.925}{S}{V}\rc{}=\comV g\lc\comV T,\comV S\rc{}  
\end{align}

\begin{reapply}
\end{reapply}

\noindent for all $g\in L^{\infty}\lc\spec T\times S,dE_{T,S}\rc$,

\item[3)] commutative diagram of normal unital surjective $^{*}$-homomorphisms

\begin{equation}\label{EQ.LEM.Compression_Preservation_II_2}
\begin{tikzcd}
L^{\infty}\lc\spec T\times S,dE_{T,S}\rc\arrow[rr,"\Gamma_{T,S}"]\arrow[dd,"\res"] & & W^{*}\lc{}T,S\rc\arrow[dd,"\comV"] \\
& & \\
L^{\infty}\lc\spec \restr{0.925}{T}{V}\times \restr{0.925}{S}{V},dE_{\restr{0.925}{T}{V},\restr{0.925}{S}{V}}\rc\arrow[rr,"\Gamma_{\restr{0.925}{T}{V},\restr{0.925}{S}{V}}"] & & W^{*}\lc\restr{0.925}{T}{V},\restr{0.925}{S}{V}\rc{}
\end{tikzcd}
\end{equation}

\begin{reapply}
\end{reapply}

\noindent with $\res$ the restriction map given by $\spec T_{1}\times S_{1}\subset\spec T_{0}\times S_{0}$.
\end{itemize}
\end{lem}
\begin{proof}
Let $T,S\in\UBII_{V}(H)$ commute strongly. Apply $2)$ in Corollary \ref{COR.Compression_Preservation_I}, which uses $C^{*}$-algebra direct sum, to show $\restr{0.925}{T}{V},\restr{0.925}{S}{V}\in\UBII\lc{}V\rc_{h}$ commute strongly. Lemma \ref{LEM.Compression_Preservation_I} shows our claims follow from Lemma \ref{LEM.JFC_Preservation} if

\begin{align}\label{EQ.LEM.Compression_Preservation_II_3}
\comV=\comV\otimes\comV    
\end{align}

\noindent on $W^{*}\lc{}T,S\rc{}=W^{*}(T)\otimes W^{*}(S)$ as per Equation \ref{EQ.SSEC.A_Fnd_FC_8}. Since Proposition \ref{PRP.Compression_Concrete} ensures $\comV$ has normal extension from $W^{*}(T)\odot W^{*}(S)$, we directly verify Equation \ref{EQ.LEM.Compression_Preservation_II_3} on elementary tensors using $2)$ in Corollary \ref{COR.Compression_Preservation_I}.
\end{proof}


\subsubsection*{Spectral gaps}

Lemma \ref{LEM.Spectral_Gap_USC} states spectral gaps are upper semi-continuous in strong resolvent convergence. Corollary \ref{COR.Spectral_Gap_USC} shows spectral gaps of local positive unbounded operators are limits of spectral gaps of compressions.\par
Let $H$ be a Hilbert space.

\begin{prp}\label{PRP.Unbd_PO_Inversion}
Let $T\geq S\geq 0$ in $\UBII(H)$. If $S$ is injective, then $T$ is injective and $S^{-1}\geq T^{-1}\geq 0$ in $\UBII(H)$.
\end{prp}
\begin{proof}
Let $L\in\UBII(H)_{+}$ be injective. For all $\varepsilon_{1}\geq\varepsilon_{0}>0$ in $\mathbb{R}$, functional calculus yields 

\begin{align}\label{EQ.PRP.Unbd_PO_Inversion_1}
0<R_{-\varepsilon_{1}}(L)\leq R_{-\varepsilon_{0}}(L)\leq L^{-1}.
\end{align}

\noindent Note Equation \ref{EQ.PRP.Unbd_PO_Inversion_1} gives monotonically increasing sequence $\lset{}R_{-n^{-1}}(L)\rset_{n\in\mathbb{N}}\subset\BII(H)$ of uniformly positive and bounded operators.\par
The Kato-Robinson theorem \lc{}cf.~Theorem 10.4.2 in \cite{BK.deOli.2009.OpAlg_Quantum_Dynamics}\rc{} shows

\begin{align}\label{EQ.PRP.Unbd_PO_Inversion_2}
L^{-1}=\sr\textrm{-}\lim_{n\in\mathbb{N}}\hspace{0.025cm} R_{-n^{-1}}(L),
\end{align}

\noindent and we obtain unique closed positive unbounded quadratic form

\begin{align}\label{EQ.PRP.Unbd_PO_Inversion_3}
u\mapsto\dblv{}\sqrt{L^{-1}}(u)\dblv_{H}^{2}=\sup_{n\in\mathbb{N}}\hspace{0.025cm} \lgl R_{-n^{-1}}(L)(u),u\rgl_{H}\in [0,\infty]
\end{align}

\noindent on $H$ represented by $L^{-1}$.\par
Let $S$ be injective. Then $T$ is injective by partial order. Applying Equation \ref{EQ.PRP.Unbd_PO_Inversion_2} and Equation \ref{EQ.PRP.Unbd_PO_Inversion_3} to $T$, resp.~$S$, we calculate

\begin{align*}
\dblv{}\sqrt{T^{-1}}(u)\dblv_{H}^{2} & = \sup_{n\in\mathbb{N}}\hspace{0.025cm} \lgl R_{-n^{-1}}(T)(u),u\rgl_{H} \phantom{\bigg)} \\
& \leq\ \sup_{n\in\mathbb{N}}\hspace{0.025cm} \lgl R_{-n^{-1}}(S)(u),u\rgl_{H} \phantom{\bigg)} \\
& = \dblv{}\sqrt{S^{-1}}(u)\dblv_{H}^{2} \phantom{\bigg)}
\end{align*}

\noindent for all $u\in H$. The above calculation implies Theorem 9.3.7 in \cite{BK.deOli.2009.OpAlg_Quantum_Dynamics} yields $S^{-1}\geq T^{-1}$.
\end{proof}

For all $T\in\UBII(H)_{+}$, we know $\spec T\subset [0,\infty)$ by definition of partial order.

\begin{dfn}\label{DFN.Spectral_Gap}
Let $T\in\UBII(H)_{+}$.

\begin{itemize}
\item[1)] The spectral gap of $T$ is $\sigma(T):=\inf\hspace{0.025cm} \big\{\hspace{0.025cm} \lambda>0\ \vset\ \lambda\in\spec T\hspace{0.025cm} \big\}$.

\item[2)] We say that $T$ has spectral gap if $\sigma(T)>0$.
\end{itemize}
\end{dfn}

\begin{prp}\label{PRP.Spectral_Gap}
If $T\in\UBII_{V}(H)$, then $\sigma(T)\leq\sigma(\restr{0.925}{T}{V})$.
\end{prp}
\begin{proof}
Apply $1)$ in Lemma \ref{LEM.Compression_Preservation_I}.
\end{proof}


\pagebreak


\begin{lem}\label{LEM.Spectral_Gap_Maximal}
For all $T\in\UBII(H)_{+}$, we have

\begin{align}\label{EQ.LEM.Spectral_Gap_Maximal_1}
\sigma(T)=\sup\hspace{0.025cm} \lset\lambda\geq 0\ \vset\ \restr{0.925}{T}{\overline{\im T}}\geq\lambda\pi_{\overline{\im T}}\rset{}.
\end{align}
\end{lem}
\begin{proof}
We write $\overline{\im T}=\overline{\im T}^{\|.\|_{H}}$ as per Example \ref{BSP.Reducible}. Note $\pi_{\overline{\im T}}=E_{T}\lc{}(0,\infty)\rc$ as $\restr{0.925}{T}{\overline{\im T}}$ is injective by Proposition \ref{PRP.FC_Injectivity}. For all $Z\in\mathfrak{B}(\mathbb{R})$, Lemma \ref{LEM.Compression_Preservation_I} implies

\begin{align}\label{EQ.LEM.Spectral_Gap_Maximal_2}
E_{\restr{0.925}{T}{\overline{\im T}}}(Z)=E_{T}\lc{}(0,\infty)\rc\cdot E_{T}(Z)\cdot E_{T}\lc{}(0,\infty)\rc{}=E_{T}\lc{}Z\cap (0,\infty)\rc{}.
\end{align}

\noindent Positivity ensures $\supp E_{T}=\supp T\subset [0,\infty)$ and $\supp E_{\restr{0.925}{T}{\overline{\im T}}}=\spec \restr{0.925}{T}{\overline{\im T}}\subset [0,\infty)$. Then Equation \ref{EQ.LEM.Spectral_Gap_Maximal_2} implies we have $\spec \restr{0.925}{T}{\overline{\im T}}=\spec T$ if and only if $\sigma(T)=0$, as well as $\spec \restr{0.925}{T}{\overline{\im T}}=\spec T\setminus\lset{}0\rset$ if and only if $\sigma(T)>0$. Set

\begin{align}\label{EQ.LEM.Spectral_Gap_Maximal_3}
\zeta(T):=\sup\hspace{0.025cm} \lset\lambda\geq 0\ \vset\ \restr{0.925}{T}{\overline{\im T}}\geq\lambda\pi_{\overline{\im T}}\rset{}.   
\end{align}

Assume $\sigma(T)=0$. Thus $0\in\spec T=\spec \restr{0.925}{T}{\overline{\im T}}$ by $\spec T$ closed. If $\zeta(T)>0$, then there exists $\lambda>0$ s.t.~

\begin{align}\label{EQ.LEM.Spectral_Gap_Maximal_4}
\restr{0.925}{T}{\overline{\im T}}\geq\lambda\pi_{\overline{\im T}}>0
\end{align}

\noindent in $\UBII\lc\overline{\im T}\rc$. Get $0\notin\spec\restr{0.925}{T}{\overline{\im T}}$ by Proposition \ref{PRP.Unbd_PO_Inversion} and Equation \ref{EQ.LEM.Spectral_Gap_Maximal_4}. We obtain $0=\sigma(T)=\zeta(T)$ as claimed.\par
Assume $\sigma(T)>0$. Since their spectra are closed, positive unbounded operators are injective with closed image if and only if they are bounded from below. Thus $\im T$ is closed as $\restr{0.925}{T}{\overline{\im T}}$ is positive and injective. Closedness further shows $\sigma(T)\in\spec T$.\par
Get $\rsl\restr{0.925}{T}{\im T}=\rsl T\cup\lset{}0\rset$ by $\spec \restr{0.925}{T}{\im T}=\spec T\setminus\lset{}0\rset$. Thus $[0,\sigma(T))\subset\rsl \restr{0.925}{T}{\im T}$, hence $\supp E_{\restr{0.925}{T}{\im T}}=\spec \restr{0.925}{T}{\im T}$ shows we have $\id_{\mathbb{R}}-\lambda\geq 0$ $E_{\restr{0.925}{T}{\im T}}$-a.e.~for all $\lambda\in [0,\sigma(T))$. We see $\restr{0.925}{T}{\im T}\geq\lambda\pi_{\im T}$ for all $\lambda\in [0,\sigma(T))$ by functional calculus and therefore $\sigma(T)\leq\zeta(T)$ by continuity. We show the converse. For all $\lambda\in \lb{}0,\zeta(T)\rc$, Equation \ref{EQ.LEM.Spectral_Gap_Maximal_3} implies

\begin{align}\label{EQ.LEM.Spectral_Gap_Maximal_5}
\restr{0.925}{T}{\im T}\geq\lambda\pi_{\im T}.
\end{align}

\noindent We claim $\lb{}0,\zeta(T)\rc\subset\rsl T$. If this holds, then $\sigma(T)\geq\zeta(T)$. Let $\lambda\in \lb{}0,\zeta(T)\rc$. Since we have $0<\sigma(T)\leq\zeta(T)$, there exists $\delta>0$ s.t.~$\lambda+\delta<\zeta(T)$. Equation \ref{EQ.LEM.Spectral_Gap_Maximal_5} shows

\begin{align}\label{EQ.LEM.Spectral_Gap_Maximal_6}
\restr{0.925}{T}{\im T}\geq(\lambda+\delta)\cdot \pi_{\im T}.
\end{align}

\noindent Subtracting $\lambda\pi_{\im T}$ in Equation \ref{EQ.LEM.Spectral_Gap_Maximal_6} shows $\restr{0.925}{T}{\im T}-\lambda\pi_{\im T}\geq\delta\pi_{\im T}$ for $\delta>0$ and $\pi_{\im T}=I_{\im T}$. Thus $\restr{0.925}{T}{\im T}-\lambda\pi_{\im T}>0$ in $\UBII\lc\im T\rc$, hence $R_{\lambda}\lc\restr{0.925}{T}{\im T}\rc\in\BII\lc\im T\rc$ as well. Using $\rsl \restr{0.925}{T}{\im T}=\rsl T\cup\lset{}0\rset$, we see $\lambda\in\rsl T$ since $\lambda>0$.
\end{proof}

\begin{cor}\label{COR.Spectral_Gap_Maximal}
For all $T\in\UBII(H)_{+}$, we either have

\begin{itemize}
\item[1)] $\sigma(T)=0$ and $\overline{\im T}\neq\im T$,

\item[2)] or $\sigma(T)>0$ and $\overline{\im T}=\im T$.
\end{itemize}
\end{cor}
\begin{proof}
Since $\spec \restr{0.925}{T}{\overline{\im T}}$ is closed, the injective positive unbounded operator $\restr{0.925}{T}{\overline{\im T}}$ has closed image if and only if it is bounded from below. Apply Lemma \ref{LEM.Spectral_Gap_Maximal}.
\end{proof}

\begin{lem}\label{LEM.Spectral_Gap_USC}
Let $T=\sr$-$\lim_{n\in\mathbb{N}}T_{n}$ on $H$. If $\lset{}T_{n}\rset_{n\in\mathbb{N}}\subset\UBII(H)_{+}$, then $T\in\UBII(H)_{+}$ and we have

\begin{align}\label{EQ.LEM.Spectral_Gap_USC_1}
\limsup_{n\in\mathbb{N}}\hspace{0.025cm} \sigma(T_{n})\leq\sigma(T).
\end{align}
\end{lem}
\begin{proof}
Corollary 10.2.2 in \cite{BK.deOli.2009.OpAlg_Quantum_Dynamics} implies $T\in\UBII(H)_{+}$. If $\limsup_{n\in\mathbb{N}}\sigma(T_{n})=0$, then our claim follows. We assume $\limsup_{n\in\mathbb{N}}\sigma(T_{n})>0$ without loss of generality. Let $\lset\sigma\lc{}T_{n_{k}}\rc\rset_{k\in\mathbb{N}}$ be a converging subsequence s.t.~$\lambda:=\lim_{k\in\mathbb{N}}\sigma\lc{}T_{n_{k}}\rc{}>0$. For all $\varepsilon\in (0,\lambda)$, let $k_{\varepsilon}\in\mathbb{N}$ s.t.~$\lset\sigma\lc{}T_{n_{k}}\rc\rset_{k\geq k_{\varepsilon}}\subset \lc\lambda-\varepsilon,\lambda+\varepsilon\rc$. Get $\lc{}0,\lambda-\varepsilon\rc\subset\bigcap_{k\geq k_{\varepsilon}}\rsl T_{n_{k}}$. Theorem 10.2.1 in \cite{BK.deOli.2009.OpAlg_Quantum_Dynamics} applies to this inclusion as $T=\sr$-$\lim_{k\in\mathbb{N}}T_{n_{k}}$, implying $\lc{}0,\lambda-\varepsilon\rc\subset\rsl T$ for all $\varepsilon\in (0,\lambda)$. Letting $\varepsilon\downarrow 0$ shows $(0,\lambda)\subset\rsl T$. Altogether, we estimate $\lambda\leq\sigma(T)$ for all non-zero accumulation points $\lambda$ of $\lset\sigma(T_{n})\rset_{n\in\mathbb{N}}\subset [0,\infty)$. This is Equation \ref{EQ.LEM.Spectral_Gap_USC_1}.
\end{proof}

\begin{cor}\label{COR.Spectral_Gap_USC}
Let $H_{1}\subset H_{2}\subset\ldots\subset H$ be Hilbert subspaces s.t.~$H=\overline{\bigcup_{n\in\mathbb{N}}H_{n}}^{\|.\|_{H}}$. If $T\in\UBII(H)_{+}$ is $H_{n}$-reducible for all $n\in\mathbb{N}$, then $T=\sr$-$\lim_{n\in\mathbb{N}}\com_{H_{n}}T$ and we have

\begin{itemize}
\item[1)] $\com_{H_{n}} T\in\BII(H)_{+}$, $\restr{0.925}{T}{H_{n}}\in\BII\lc{}H_{n}\rc_{+}$ and $\sigma\lc\com_{H_{n}}T\rc{}=\sigma\lc\restr{0.925}{T}{H_{n}}\rc$ for all $n\in\mathbb{N}$,

\item[2)] monotonically decreasing sequence $\big\{\sigma\lc\restr{0.925}{T}{H_{n}}\rc\big\}_{n\in\mathbb{N}}\subset [0,\infty)$,

\item[3)] $\sigma(T)=\lim_{n\in\mathbb{N}}\sigma\lc\restr{0.925}{T}{H_{n}}\rc$.
\end{itemize}
\end{cor}
\begin{proof}
Let $T\in\UBII(H)_{+}$ be $H_{n}$-reducible for all $n\in\mathbb{N}$. Using $2)$ in Corollary \ref{COR.Compression_Preservation_I} and $I_{H}=\s$-$\lim_{n\in\mathbb{N}}\pi_{H_{n}}$, get $T=\sr$-$\lim_{n\in\mathbb{N}}\com_{H_{n}}T$. Moreover, we see $T\in\UBII(H)_{+}$ and $1.1)$ in Proposition \ref{PRP.Reducible} show $\com_{H_{n}}T\in\BII(H)_{+}$ and $\restr{0.925}{T}{H_{n}}\in\BII\lc{}H_{n}\rc_{+}$ as per Notation \ref{NTN.Reducible} for all $n\in\mathbb{N}$. Proposition \ref{PRP.Spectral_Gap} and Lemma \ref{LEM.Spectral_Gap_USC} thus imply our claims if

\begin{align}\label{EQ.COR.Spectral_Gap_USC_1}
\sigma\lc\textrm{com}_{H_{n}}\hspace{0.025cm} T\rc{}=\sigma\lc\restr{0.925}{T}{H_{n}}\rc{}
\end{align}

\noindent for all $n\in\mathbb{N}$. Let $n\in\mathbb{N}$. For all $\lambda\in\mathbb{R}$, we decompose

\begin{align}\label{EQ.COR.Spectral_Gap_USC_2}
\textrm{com}_{H_{n}}\hspace{0.025cm} T-\lambda I_{H}=\lc\restr{0.925}{T}{H_{n}}-\lambda I_{H_{n}}\rc\oplus -\lambda I_{H_{n}^{\perp}}
\end{align}

\noindent w.r.t.~$\BII\lc{}H_{n}\rc\oplus\BII\lc{}H_{n}^{\perp}\rc$. If $H_{n}=H$, then there is nothing to show. We assume $H_{n}\neq 0$ without loss of generality. Using decomposition as per Equation \ref{EQ.COR.Spectral_Gap_USC_2}, we directly verify Equation \ref{EQ.COR.Spectral_Gap_USC_1} by definition of spectra.
\end{proof}


\chapter{Noncommutative Measure and Integration Theory}\label{APP.B}

Theorem \ref{THM.JFC_Compression} states sufficient conditions for compressing joint functional calculus pulled-back to joint functional calculus of self-adjoint measurable operators. The latter are noncommutative measurable functions. Tracial $W^{*}$-algebras define such spaces of measurable operators. For all $p\in [1,\infty]$, we define noncommutative $L^{p}$-spaces of measurable operators equipped with $L^{p}$-norm \cite{BK.Joh_Lin.2003.Wstar_Lp_II}\cite{ART.Nel.1974.Wstar_Integration}. They fulfil H\"older inequalities. We have a modified standard pairing encoding duality \cite{BK.Tak.2003.OpAlg_II}.\par
In Section \ref{SEC.B_SMO}, we discuss tracial $W^{*}$-algebras, spaces of measurable operators and noncommutative integration theory. We study canonical left-~and right-actions of spaces of measurable operators. In Section \ref{SEC.B_JFC}, we prove Theorem \ref{THM.JFC_Compression} using compression maps given by change of canonical left-~and right-actions. We formulate compressed pulled-backed joint functional calculus of self-adjoint measurable operators.


\section{Spaces of measurable operators}\label{SEC.B_SMO}

In Subsection \ref{SSEC.B_SMO_Wstar_Trace}, we discuss tracial $C^{*}$-~and $W^{*}$-algebras. The GNS-construction for traces defines canonical left-actions. Each is a faithful normal unital $^{*}$-representation over noncommutative $L^{2}$-space, i.e.~Hilbert space given by GNS-construction. Tracial $C^{*}$-algebras are a preliminary step useful for the AF-$C^{*}$-setting. Canonical right-actions are canonical left-actions of opposite tracial $W^{*}$-algebras.\par
In Subsection \ref{SSEC.B_SMO_NCI}, we discuss spaces of measurable operators and noncommutative integration theory. Spaces of measurable operators are uniformly completed $^{*}$-algebras in measure topologies of tracial $W^{*}$-algebras. Traces extend. For all $p\in [1,\infty]$, we define noncommutative $L^{p}$-spaces via $L^{p}$-norms using traces of measurable operators. H\"older inequalities apply and we have a modified standard pairing.\par
In Subsection \ref{SSEC.B_SMO_CLRA}, we further extend canonical left-~and right-actions to spaces of measurable operators using $^{*}$-algebra multiplication. We account for noncommutative $L^{2}$-spaces different from Hilbert spaces given by GNS-construction. Whereas canonical left-actions represent $^{*}$-algebras of measurable operators, canonical right-actions are twisted canonical left-actions defined on opposite $^{*}$-algebras. Using canonical left-~and right-actions, we define spectral and joint spectral measures of self-adjoint measurable operators. This lets us formulate their bounded measurable joint functional calculus.\par


\newpage



\subsection[Tracial $C^{*}$-~and $W^{*}$-algebras]{Tracial $\mathbf{C}^{*}$-~and $\mathbf{W}^{*}$-algebras}\label{SSEC.B_SMO_Wstar_Trace}

Tracial $W^{*}$-algebras have f.s.n.~traces. Applying GNS-construction, each is represented over noncommutative $L^{2}$-space via canonical left-actions. Canonical right-actions arise using opposite tracial $W^{*}$-algebras. Remark \ref{REM.Wstar_CLRA} explains there is no twisting in the bounded case. Standard references for tracial $C^{*}$-~and $W^{*}$-algebras are \cite{BK.Dix.1977.Cstar_Algebras} and \cite{BK.Tak.1979.OpAlg_I}\linebreak\cite{BK.Tak.2003.OpAlg_II}. Note \cite{BK.Tak.2003.OpAlg_II} discusses general weights as an extension of the tracial case.


\subsubsection*{Tracial $\mathbf{W^{*}}$-algebras and canonical left-actions}

In Subsection \ref{SSEC.A_Fnd_CWstar}, we cover $C^{*}$-~and $W^{*}$-algebras. Definition \ref{DFN.Cstar_PO} fixes partial orders.

\begin{dfn}\label{DFN.Wstar_Trace}
Let $A$ be a $C^{*}$-algebra. Set $\infty\cdot 0=0\cdot \infty=0$ as convention.

\begin{itemize}
\item[1)] A map $\tau:A_{+}\longrightarrow [0,\infty]$ is a trace on $A$ if

\begin{itemize}
\item[1.1)] $\tau\lc{}x+y\rc{}=\tau(x)+\tau(y)$ for all $x,y\in A_{+}$, \phantom{\big)} \hfill (Linearity)

\item[1.2)] $\tau\lc\lambda x\rc{}=\lambda\tau(x)$ for all $x\in A_{+}$ and $\lambda\geq 0$, \phantom{\big)} \hfill (Homogeneity)

\item[1.3)] $\tau(x^{*}x)=\tau\lc{}xx^{*}\rc$ for all $x\in A$. \phantom{\big)} \hfill (Traciality)
\end{itemize}

\begin{reapply}
\end{reapply}

\item[2)] Let $\tau$ be a trace on $A$. We say that $\tau$ is

\begin{itemize}
\item[2.1)] l.s.c.~if it is l.s.c.~in $\|.\|_{A}$, \phantom{\big)}

\item[2.2)] faithful if $\tau(x)=0$ implies $x=0$ for all $x\in A_{+}$, \phantom{\big)}

\item[2.3)] semi-finite if $\tau(x)=\sup\hspace{0.025cm} \big\{\hspace{0.025cm} \tau(y)\ \vset\ y\in A_{+}:\ y\leq x,\ \tau(y)<\infty\hspace{0.025cm} \big\}$ for all $x\in A_{+}$. \phantom{\big)}
\end{itemize}

\begin{reapply}
\end{reapply}

\item[3)] Let $\tau$ be a faithful trace on $A$. Set $\mathfrak{n}_{\tau}:=\big\{\hspace{0.025cm} x\in A\ \vset\ \tau(x^{*}x)<\infty\hspace{0.025cm} \big\}$. We call

\begin{align}\label{EQ.DFN.Wstar_Trace_1}
\mathfrak{m}_{\tau}:=\lset{} x\in A\ \vset\ \exists\{\hspace{0.0125cm} y_{k}\hspace{0.0025cm}\}_{k=1}^{n},\{\hspace{0.025cm}z_{k}\hspace{0.025cm}\}_{k=1}^{n}\subset\mathfrak{n}_{\tau}:\ x=\sum_{k=1}^{n}y_{k}^{*}z_{k}\rset{}
\end{align}

\begin{reapply}
\end{reapply}

\noindent the definition domain of $\tau$.

\item[4)] We call $(A,\tau)$ a tracial $C^{*}$-algebra if $\tau$ is a l.s.c.~faithful semi-finite trace on $A$.

\item[5)] Let $(A,\tau)$ be a tracial $C^{*}$-algebra and $\phi:\mathfrak{m}_{\tau}\longrightarrow A$. We say that $\phi$ is a dilation if $0\leq\tau(\phi(x))\leq\tau(x)$ for all $x\in\mathfrak{m}_{\tau}\cap A_{+}$. We call $\phi$ trace-, or $\tau$-preserving if $\tau(\phi(x))=\tau(x)$ for all $x\in\mathfrak{m}_{\tau}$. 
\end{itemize}
\end{dfn}

Let $A$ be a $C^{*}$-algebra and $\tau$ a faithful trace on $A$. Note $\mathfrak{n}_{\tau},\mathfrak{m}_{\tau}=\mathfrak{n}_{\tau}^{2}=\langle\mathfrak{m}_{\tau}\cap A_{+}\rangle_{\mathbb{C}}\subset A$ are self-adjoint two-sided ideals \lc{}cf.~Lemma 4.5.1 and Proposition 6.1.2 in \cite{BK.Dix.1977.Cstar_Algebras}\rc{}. There exists unique linear extension of $\tau$ to $\mathfrak{m}_{\tau}$ since $\tau\lc\mathfrak{m}_{\tau}\cap A_{+}\rc{}<\infty$ \lc{}cf.~Proposition 6.1.2 in \cite{BK.Dix.1977.Cstar_Algebras}\rc{}. We denote extension by $\tau$. For all $x,y\in\mathfrak{m}_{\tau}$, $\absv{1.15}{\tau(x)}<\infty$ and $\tau\lc{}xy\rc{}=\tau\lc{}yx\rc$. The notion of $\tau$-preserving map as per $5)$ in Definition \ref{DFN.Wstar_Trace} is well-defined.

\begin{rem}\label{REM.Wstar_Trace_PO}
For all $x,y\in\mathfrak{m}_{\tau}$ self-adjoint, $x\geq y$ implies $\tau(x)\geq\tau(y)$. In this sense, $\tau$ is positivity-preserving. This corresponds to Definition \ref{DFN.Cstar_PO}.
\end{rem}


\pagebreak


Let $(A,\tau)$ be a tracial $C^{*}$-algebra. The GNS-inner product of $\tau$ on $\mathfrak{n}_{\tau}$ is given by

\begin{align}\label{EQ.SSEC.B_SMO_Wstar_Trace_1}
\lgl x,y\rgl_{\tau}:=\tau(x^{*}y)    
\end{align}

\noindent for all $x,y\in\mathfrak{n}_{\tau}$. For all $x\in\mathfrak{n}_{\tau}$, faithfulness shows $\|x\|_{\tau}=0$ if and only if $x=0$. Its Hilbert space completion is noncommutative $L^{2}$-space $\HII(A,\tau)$. For all $x\in A$, set $\mathcal{L}_{x}(y):=xy\in\mathfrak{n}_{\tau}$ for all $y\in\mathfrak{n}_{\tau}$ and extend to $\mathcal{L}_{x}\in\BII\lc\HII(A,\tau)\rc$. This is the GNS-construction for $\tau$. Thus $\mathcal{L}$ is a semi-cyclic $^{*}$-representation \lc{}cf.~Theorem I.9.14 in \cite{BK.Tak.1979.OpAlg_I} and Definition VII.1.5 in \cite{BK.Tak.2003.OpAlg_II}\rc{}, hence a faithful $^{*}$-representation of $A$ over $\HII(A,\tau)$. It is non-degenerate by l.s.c.~\lc{}cf.~Lemma VII.4.1 in \cite{BK.Tak.2003.OpAlg_II}\rc{}. We see unitality of $A$ implies that of $\mathcal{L}$.

\begin{dfn}\label{DFN.Wstar_Trace_CLA}
For all tracial $C^{*}$-algebras $(A,\tau)$, we call $\HII(A,\tau):=\overline{\mathfrak{n}}_{\tau}^{\|.\|_{\tau}}$ the concrete noncommutative $L^{2}$-space and $\mathcal{L}$ the canonical left-action of $A$ on $\HII(A,\tau)$.
\end{dfn}

\begin{rem}
Canonical right-actions are given in Definition \ref{DFN.Wstar_Trace_CRA}. For this, we use the opposite $^{*}$-algebra construction given in Definition \ref{DFN.Oppalg}. Note Definition \ref{DFN.Wstar_CLRA} subsumes canonical left-~and right-actions in this subsection.
\end{rem}

Traces on $W^{*}$-algebras must have canonical normal left-action in order to preserve bounded measurable functional calculus. Faithful, semi-finite and normal traces, or f.s.n.~traces on $W^{*}$-algebras have canonical normal left-action. Tracial $W^{*}$-algebras are all $W^{*}$-algebras equipped with an f.s.n.~trace.\par
Proposition \ref{PRP.Wstar_Equivalence}, fundamentally a useful reformulation of the double commutant theorem \cite{BK.Kad_Rin.1997.OpAlg_II}\cite{BK.Tak.1979.OpAlg_I}, states double commutants of concrete $C^{*}$-algebra are, up to normal faithful unital $^{*}$-representations, all $W^{*}$-algebras. Proposition \ref{PRP.Wstar_Trace_Ext_I} implies each tracial $C^{*}$-algebra induces unique f.s.n.~extension of their trace to the double commutant of their image $C^{*}$-algebra. Finally, Proposition \ref{PRP.Wstar_Trace_Ext_II} ensures each tracial $W^{*}$-algebra is a tracial $C^{*}$-algebra with image $C^{*}$-algebra being its own double commutant. We thereby reduce from tracial $C^{*}$-~to tracial $W^{*}$-algebras.

\begin{dfn}\label{DFN.Wstar_Trace_FSN}
Let $M$ be a $W^{*}$-algebra.

\begin{itemize}
\item[1)] A trace $\tau$ on $M$ is normal if for all bounded increasing nets $\{x_{k}\}_{k\in K}\subset M_{+}$, get

\begin{align}\label{EQ.DFN.Wstar_Trace_FSN_1}
\tau\vstretch{1.05}{\Bigg(}\hspace{0.0075cm} \sup_{k\in K}\hspace{0.025cm} x_{k}\vstretch{1.05}{\Bigg)}=\sup_{k\in K}\hspace{0.025cm} \tau(x_{k}).
\end{align}

\begin{reapply}
\end{reapply}

\item[2)] A trace $\tau$ on $M$ is f.s.n.~if it is faithful, semi-finite and normal.

\item[3)] We call $(M,\tau)$ a tracial $W^{*}$-algebra if $\tau$ is a f.s.n.~trace on $M$.
\end{itemize}
\end{dfn}

\begin{rem}
Equation \ref{EQ.DFN.Wstar_Trace_FSN_1} corresponds to Equation \ref{EQ.DFN.Wstar_Normal_1}, i.e.~normality for bounded linear maps of $W^{*}$-algebras. The two notions coincide assuming boundedness.
\end{rem}


\pagebreak


\begin{prp}\label{PRP.Wstar_Trace_Ext_I}
Let $(A,\tau)$ be a tracial $C^{*}$-algebra.

\begin{itemize}
\item[1)] There exists unique f.s.n.~trace $\tau_{\infty}$ on $\LII(A)''$ extending $\tau$ from $A_{+}$ to $\LII(A)_{+}''$.

\item[2)] $\lc\LII(A)'',\tau_{\infty}\rc$ is a tracial $W^{*}$-algebra.
\end{itemize}
\end{prp}
\begin{proof}
Get $\tau_{\infty}$ by applying Lemma 6.1.5 in \cite{BK.Dix.1977.Cstar_Algebras} to $\mathcal{L}$ \lc{}cf.~A.60 in \cite{BK.Dix.1977.Cstar_Algebras}\rc{}. We have $\tau_{\infty}=\tau$ on $A_{+}\subset \LII(A)_{+}''$ by Proposition 6.6.5 in \cite{BK.Dix.1977.Cstar_Algebras}. Thus $\lc\LII(A)'',\tau_{\infty}\rc$ is tracial $W^{*}$-algebra.
\end{proof}

\begin{ntn}\label{NTN.Wstar_Trace}
Let $(A,\tau)$ be a tracial $C^{*}$-algebra. We write $\tau=\tau_{\infty}$ on $\LII(A)''$.
\end{ntn}

\begin{prp}\label{PRP.Wstar_Trace_Ext_II}
Let $(M,\tau)$ be a tracial $W^{*}$-algebra.

\begin{itemize}
\item[1)] $(M,\tau)$ is a tracial $C^{*}$-algebra and $\mathcal{L}$ is faithful normal unital $^{*}$-representation s.t.~$w^{*}$-topology on $M$ maps to $\sigma$-weak topology on $\mathcal{L}(M)$,

\item[2)] $\lc\LII(M)'',\tau\rc{}=\lc\mathcal{L}(M),\tau\rc$.
\end{itemize}
\end{prp}
\begin{proof}
Normality of $\tau$ shows l.s.c.~in $\sigma$-weak topology by Theorem VII.1.11 in \cite{BK.Tak.2003.OpAlg_II}. Thus $\tau$ is l.s.c. in norm, hence $(M,\tau)$ is tracial $C^{*}$-algebra and we know $\mathcal{L}$ is faithful unital $^{*}$-representation. Equation \ref{EQ.DFN.Wstar_Trace_FSN_1} shows normality of $\mathcal{L}$. Its construction and normality then show $\mathcal{L}$ maps $w^{*}$-topology on $M$ to $\sigma$-weak topology on $\mathcal{L}(M)$. Proposition \ref{PRP.Wstar_Equivalence} implies $\mathcal{L}(M)=\mathcal{L}(M)''$ at once.
\end{proof}

\begin{prp}\label{PRP.Wstar_Trace_Ext_III}
If $(A,\tau)$ is a tracial $C^{*}$-algebra, then $\HII\lc\LII(A)'',\tau\rc{}=\HII\lc\LII(A),\tau\rc$.
\end{prp}
\begin{proof}
Apply Proposition \ref{PRP.Wstar_Trace_Ext_I} and Proposition \ref{PRP.Wstar_Trace_Ext_II}.
\end{proof}

Finite faithful traces on unital $C^{*}$-algebras are well-behaved.

\begin{dfn}
Let $A$ be a $C^{*}$-algebra and $\tau$ a trace on $A$. We call $\tau$ finite if $\tau(x)<\infty$ for all $x\in A_{+}$ and further write $\tau<\infty$.
\end{dfn}

\begin{prp}\label{PRP.Wstar_Trace_Fin_I}
Let $A$ be a unital $C^{*}$-algebra and $\tau$ a faithful trace on $A$.

\begin{itemize}
\item[1)] $\tau<\infty$ if and only if $\tau(1_{A})<\infty$.

\item[2)] If $\tau<\infty$, then $\tau$ is semi-finite.

\item[3)] If $\tau<\infty$, then $(A,\tau)$ is a tracial $C^{*}$-algebra, $\tau\in A_{+}^{*}$ and $\ker\tau^{\perp}=\langle 1_{A}\rangle_{\mathbb{C}}$.
\end{itemize}
\end{prp}
\begin{proof}
If $\tau<\infty$, then $\tau(1_{A})<\infty$. Assume $\tau(1_{A})<\infty$. For all $x\in A_{+}$, get $x\leq \|x\|_{A}1_{A}$ by functional calculus and therefore $\absv{1.15}{\tau(x)}\leq \|x\|_{A}\tau(1_{A})<\infty$ by positivity-preservation on $\mathfrak{m}_{\tau}$ as per Remark \ref{REM.Wstar_Trace_PO}. Get $1)$. Assume $\tau$ is finite. Thus $A=\mathfrak{m}_{\tau}$, hence $\tau\in A_{+}^{*}$. We see $\tau$ is semi-finite and l.s.c.~in norm. In particular, $(A,\tau)$ is a tracial $C^{*}$-algebra and we have $1_{A}\in \HII(A,\tau)$. Since $\dim_{\mathbb{C}}\ker\tau^{\perp}=1$ by $\tau\in A^{*}$ and $1_{A}\in\ker\tau^{\perp}$ by faithfulness, get $\ker\tau^{\perp}=\langle 1_{A}\rangle_{\mathbb{C}}\subset \HII(A,\tau)$. Altogether, get $2)$ and $3)$.
\end{proof}

\begin{prp}\label{PRP.Wstar_Trace_Fin_II}
Let $M$ be a $W^{*}$-algebra and $\tau$ a faithful normal trace on $M$. If $\tau<\infty$, then $\tau$ is f.s.n.~trace on $M$.
\end{prp}
\begin{proof}
Apply $2)$ in Proposition \ref{PRP.Wstar_Trace_Fin_I}.
\end{proof}


\subsubsection*{Opposite tracial $W^{*}$-algebras and canonical right-actions}

Proposition \ref{PRP.Wstar_Trace_Ext_I} and Proposition \ref{PRP.Wstar_Trace_Ext_III} show canonical right-actions for tracial $C^{*}$-algebras reduce to tracial $W^{*}$-algebras. Let $(M,\tau)$ be a tracial $W^{*}$-algebra. For all $x\in M$, set $\mathcal{R}_{x}(y):=yx\in\mathfrak{n}_{\tau}$ for all $y\in\mathfrak{n}_{\tau}$ and extend to $\mathcal{R}_{x}\in\BII\lc\HII(M,\tau)\rc$.
 
\begin{dfn}\label{DFN.Wstar_Trace_CRA}
Let $(M,\tau)$ be a tracial $W^{*}$-algebra. Following Definition \ref{DFN.Wstar_Trace_CLA}, we call $\mathcal{R}$ the canonical right-action of $M$ on $\HII(M,\tau)$.
\end{dfn}

Definition \ref{DFN.Oppalg} gives an opposite $^{*}$-algebra construction. Proposition \ref{PRP.Wstar_Trace_CRA_I} shows $(M,\tau)$ yields opposite tracial $W^{*}$-algebra $\lc{}M^{\op},\tau\rc$ s.t.~$\HII\lc{}M^{\op},\tau\rc{}=\HII(M,\tau)$. Using the latter, Proposition \ref{PRP.Wstar_Trace_CRA_II} shows $\mathcal{R}$ is canonical left-action $\mathcal{L}^{\op}$ of $M^{\op}$ on $\HII(M,\tau)$ and Proposition \ref{PRP.Wstar_Trace_CLRA_Bd} implies our discussion concerning canonical left-actions translates to canonical right-actions as per Diagram \ref{EQ.PRP.Wstar_Trace_CLRA_Bd_1}.

\begin{dfn}\label{DFN.Oppalg}
Let $\mathcal{A}$ be a $^{*}$-algebra and $\Adj:\mathcal{A}\longrightarrow\mathcal{A}$ its algebra involution. Its opposite $^{*}$-algebra $\mathcal{A}^{\op}$ has $\mathcal{A}$ as complex vector space and is equipped with

\begin{itemize}
\item[1)] opposite algebra action given by $x\cdot^{\op}y:=yx$ for all $x,y\in\mathcal{A}$,

\item[2)] $\Adj:\mathcal{A}^{\op}\longrightarrow\mathcal{A}^{\op}$ as algebra involution.
\end{itemize}
\end{dfn}

\begin{rem}
If $\mathcal{A}$ is a topological vector space, then $\mathcal{A}^{\op}$ is one using the identical topology. For all $W^{*}$-algebras, we use identical norm and $w^{*}$-topology on opposites.
\end{rem}

\begin{prp}\label{PRP.Wstar_Trace_CRA_I}
For all tracial $W^{*}$-algebras $(M,\tau)$, we have

\begin{itemize}
\item[1)] $\tau$ is f.s.n.~trace on $M^{\op}$ and $\lc{}M^{\op},\tau\rc$ is a tracial $W^{*}$-algebra,

\item[2)] $\lc\HII\lc{}M^{\op},\tau\rc{},\|.\|_{\tau,\op}\rc{}=\lc\HII(M,\tau),\|.\|_{\tau}\rc$.
\end{itemize}
\end{prp}
\begin{proof}
Note $M^{\op}$ is a $C^{*}$-algebra with norm and algebra involution of $M$. This implies $M^{\op}=M=\lc{}M_{*}\rc^{*}$ as Banach spaces. Thus $M$ is a $W^{*}$-algebra s.t.~$M_{+}^{\op}=M_{+}$, hence $\tau$ is f.s.n.~trace on $M^{\op}$. We obtain $1)$. Traciality moreover ensures $\mathfrak{n}_{\tau}$ defined by $\tau$ on $M$ and $M^{\op}$ are identical. Get $2)$ by construction.
\end{proof}

\begin{ntn}
We write $\mathcal{L}^{\op}$ for the canonical left-action $\mathcal{L}^{\op}$ of $M^{\op}$ on $\HII(M,\tau)$. We write $\mathfrak{n}_{\tau,\op}$ as per $3)$ in Definition \ref{DFN.Wstar_Trace} for f.s.n.~trace $\tau$ on $M^{\op}$.
\end{ntn}

Proposition \ref{PRP.Wstar_Trace_CRA_II} shows $\mathcal{R}=\mathcal{L}^{\op}$ on $M^{\op}$. Note $\mathcal{R}\neq\mathcal{L}^{\op}$ in general as extensions of $M$ and $M^{\op}$ are different spaces of measurable operators. We show $\mathcal{R}\cong\mathcal{L}^{\op}$ naturally extends the bounded case. For details on the latter, we refer to Subsection \ref{SSEC.B_SMO_NCI}.

\begin{prp}\label{PRP.Wstar_Trace_CRA_II}
Let $(M,\tau)$ be a tracial $W^{*}$-algebra.

\begin{itemize}
\item[1)] $\mathcal{R}=\mathcal{L}^{\op}$ is faithful normal unital $^{*}$-representation s.t.~$w^{*}$-topology on $M^{\op}$ maps to $\sigma$-weak topology on $\mathcal{R}(M)$.

\item[2)] $\lc\RII(M^{\op})'',\tau\rc{}=\lc\LII(M)^{\op},\tau\rc$.
\end{itemize} 
\end{prp}
\begin{proof}
For $\tau$ on $M$, resp.~$M^{\op}$ traciality ensures $\mathfrak{n}_{\tau}=\mathfrak{n}_{\tau,\op}$. For all $x\in M$, we calculate $\mathcal{R}_{x}(y)=xy=y\cdot^{\op}x=\mathcal{L}_{x}^{\op}$ for all $y\in\mathfrak{n}_{\tau}$. Get $1)$ and $2)$ by Proposition \ref{PRP.Wstar_Trace_Ext_II}.
\end{proof}


\pagebreak


We have anti-linear isometric involution $\Adj:\HII(M,\tau)\longrightarrow \HII(M,\tau)$ by closing $\restr{0.925}{\Adj}{\mathfrak{n}_{\tau}}$ w.r.t~$\|.\|_{\tau}$. Get $\Adj^{\dagger}:\BII\lc\HII(M,\tau)\rc\longrightarrow\BII\lc\HII(M,\tau)\rc$ as per Definition \ref{DFN.Unbd_Twist}.

\begin{dfn}\label{DFN.Wstar_Trace_L2_Adj}
Let $(M,\tau)$ be a tracial $W^{*}$-algebra. $\Adj:\HII(M,\tau)\longrightarrow \HII(M,\tau)$ is called adjoining on $\HII(M,\tau)$.
\end{dfn}

\begin{prp}\label{PRP.Wstar_Trace_CLRA_Bd}
Let $(M,\tau)$ be a tracial $W^{*}$-algebra. We have commutative diagram

\smallskip

\begin{equation}\label{EQ.PRP.Wstar_Trace_CLRA_Bd_1}
\begin{tikzcd}
M\arrow[rr,"\mathcal{L}"]\arrow[dd,"\Adj"] & & \BII\lc\HII(M,\tau)\rc\arrow[dd,"\Adj^{\dagger}"] \\
& & \\
M^{\op}\arrow[rr,"\mathcal{R}"] & & \BII\lc\HII(M,\tau)\rc{}
\end{tikzcd}
\end{equation}

\medskip

\noindent s.t.~horizontal maps are normal unital injective $^{*}$-homomorphisms and vertical ones are isometric involutions of Banach spaces.
\end{prp}
\begin{proof}
We directly verify Diagram \ref{EQ.PRP.Wstar_Trace_CLRA_Bd_1} and all claims.
\end{proof}


\subsection[Noncommutative integration for tracial $W^{*}$-algebras]{Noncommutative integration for tracial $\mathbf{W}^{*}$-algebras}\label{SSEC.B_SMO_NCI}

We discuss spaces of measurable operators and noncommutative integration theory. Traces extend. For all $p\in [1,\infty]$, we define noncommutative $L^{p}$-spaces of measurable operators equipped with $L^{p}$-norm \cite{BK.Joh_Lin.2003.Wstar_Lp_II}\cite{ART.Nel.1974.Wstar_Integration}. They fulfil H\"older inequalities. We have a modified standard pairing encoding duality \cite{BK.Tak.2003.OpAlg_II}. In particular, tracial $W^{*}$-algebras are noncommutative $L^{\infty}$-spaces and have noncommutative $L^{1}$-spaces as pre-duals. We see their f.s.n.~traces are, possibly unbounded \cite{BK.Pap.2002.Measures}\cite{BK.Ped.1989.Analysis_Now}, noncommutative Radon measures. Standard references for their spaces of measurable operators and resulting notion of noncommutative integration are pp.1461-1470 in \cite{BK.Joh_Lin.2003.Wstar_Lp_II}, \cite{ART.Nel.1974.Wstar_Integration} and \cite{BK.Tak.1979.OpAlg_I}\cite{BK.Tak.2003.OpAlg_II}.


\subsubsection*{Spaces of measurable operators}

Let $(M,\tau)$ be a tracial $W^{*}$-algebra. Its space $L^{0}(M,\tau)$ of measurable operators is uniform completion in measure topology and serves as setting for noncommutative integration theory. For $p=\infty$, get $M\subset L^{0}(M,\tau)$. For all $p\in [1,\infty]$, get $L^{p}(M,\tau)\subset L^{0}(M,\tau)$ as per Definition \ref{DFN.NC_Int_Lp}. Uniform completion extends the $^{*}$-algebra structure and trace from $M$ to $L^{0}(M,\tau)$ as per Remark \ref{REM.Wstar_SMO_I_Extension}.\par
We thereby extend canonical left-action $\mathcal{L}:M\longrightarrow\BII\lc\HII(M,\tau)\rc$ to an unbounded faithful unital $^{*}$-representation $\mathcal{L}:L^{0}(M,\tau)\longrightarrow\UBII\lc\HII(M,\tau)\rc$. Remark \ref{REM.Wstar_SMO_CLA} explains $\mathcal{L}$ does not equal canonical left-action of measurable operators in general.

\begin{rem}\label{REM.Wstar_SMO_CLA}
If we twist $\mathcal{L}$ as per Definition \ref{DFN.Unbd_Twist} using the natural isometric isomorphism $\HII(M,\tau)\cong L^{2}(M,\tau)$ implied by Proposition \ref{PRP.Wstar_NCI_I}, then we obtain canonical left-action $L$ using left-multiplication in $L^{0}(M,\tau)$ as per Definition \ref{DFN.Wstar_CLRA_Unbd_Representation} and based on Definition \ref{DFN.Wstar_CLRA}. This subsumes canonical left-action in the bounded case.
\end{rem}

Note $P(M)$ is the set of all projections in $M$. The measure topology of $(M,\tau)$ is defined by the following fundamental system of neighbourhoods of zero. For all $\varepsilon,\delta>0$, set

\begin{align}\label{EQ.SSEC.B_SMO_NCI_1}
N\lc\varepsilon,\delta\rc{}:=\lset{}x\in M\ \vset\ \exists p\in P(M):\ \|xp\|_{M}<\varepsilon,\ \tau\lc{}p^{\perp}\rc{}<\delta\rset{}.
\end{align}

\noindent The fundamental system of entourages given by $U\lc\varepsilon,\delta\rc{}:=\lset (x,y)\in M\times M\ \vset\ x-y\in N\lc\varepsilon,\delta\rc{} \rset$ for all $\varepsilon,\delta>0$ defines uniform structure of measure topology on $M$. Convergence in measure topology is called convergence in measure.

\begin{dfn}\label{DFN.Wstar_SMO_I}
Let $L^{0}(M,\tau)$ be the uniform closure of $M$ in measure topology. We call it the space of measurable operators for $(M,\tau)$, or $\tau$-measurable operator algebra.
\end{dfn}

\begin{rem}\label{REM.Wstar_SMO_I_Extension}
Theorem IX.2.2 in \cite{BK.Tak.2003.OpAlg_II} shows the $^{*}$-algebra structure of $M$ extends to $L^{0}(M,\tau)$. Lemma IX.2.3 in \cite{BK.Tak.2003.OpAlg_II} shows $L^{0}(M,\tau)$ is Hausdorff and $M\subset L^{0}(M,\tau)$.    
\end{rem}

We additionally have measure topology on $\HII(M,\tau)$, as well as subsequent notion of convergence in measure. The measure topology of $\HII(M,\tau)$ is defined by the following fundamental system of neighbourhoods of zero. For all $\varepsilon,\delta>0$, set

\begin{align}\label{EQ.SSEC.B_SMO_NCI_2}
O\lc\varepsilon,\delta\rc{}:=\lset{}u\in \HII(M,\tau)\ \vset\ \exists p\in P(M):\ \|p(u)\|_{\tau}<\varepsilon,\ \tau\lc{}p^{\perp}\rc{}<\delta\rset{}.
\end{align}

\noindent Uniform structure of measure topology on $\HII(M,\tau)$ follows as for $L^{0}(N,\tau)$. Convergence in measure topology is called convergence in measure. $\HII(M,\tau)$ is not complete.\par
Let $x\in L^{0}(M,\tau)$. We construct densely defined closed operator $\mathcal{L}_{x}$ on $\HII(M,\tau)$. Let $\dom\mathcal{L}_{x}$ be the set of all $u\in\HII(M,\tau)$ s.t.~there exists a net $\{x_{k}\}_{k\in K}\subset M$ converging to $x$ in measure and for which $\lset{}x_{k}u\rset_{k\in K}\subset\HII(M,\tau)$ converges in measure to an element in $\HII(M,\tau)$. For all $u\in\dom\mathcal{L}_{x}$, set

\begin{align}\label{EQ.SSEC.B_SMO_NCI_3}
\mathcal{L}_{x}(u):=\lim_{k\in K}\hspace{0.025cm} \mathcal{L}_{x_{k}}(u)\in \HII(M,\tau)
\end{align}

\noindent using limit in measure topology on $\HII(M,\tau)$. Equation \ref{EQ.SSEC.B_SMO_NCI_3} defines $\mathcal{L}_{x}(u)$ independent of converging net $\{x_{k}\}_{k\in K}\subset M$. Theorem IX.2.5 in \cite{BK.Tak.2003.OpAlg_II} shows $\mathcal{L}_{x}$ is a densely defined closed operator on $\HII(M,\tau)$. Moreover, Proposition \ref{PRP.Wstar_SMO_III} implies those operators as per Equation \ref{EQ.SSEC.B_SMO_NCI_3} for all $x\in L^{0}(M,\tau)$ define unbounded faithful unital $^{*}$-representation

\begin{align}\label{EQ.SSEC.B_SMO_NCI_4}
\mathcal{L}:L^{0}(M,\tau)\longrightarrow\UBII\lc\HII(M,\tau)\rc{}.
\end{align}

\noindent Restricting to $M\subset L^{0}(M,\tau)$, we recover canonical left-action of $M$ on $\HII(M,\tau)$ as per Definition \ref{DFN.Wstar_Trace_CLA}. We understand $^{*}$-algebra structure using $\mathcal{L}$. It has image the set of all $\tau$-measurable operators on $\HII(M,\tau)$. Their definition requires the notion of $M$-affiliated operator. The commutant $\LII(M)'\subset\BII\lc\HII(M,\tau)\rc$ is a $W^{*}$-algebra.

\begin{dfn}
A densely defined closed unbounded operator $T$ on $\HII(M,\tau)$ is called $M$-affiliated if $TU=UT$ for all unitaries $U\in\UII\lc\LII(M)'\rc\subset\BII\lc\HII(M,\tau)\rc$.
\end{dfn}


\pagebreak


\begin{prp}\label{PRP.Wstar_SMO_I}
If $T\in\UBII(\HII(M,\tau))_{h}$, then we know $T$ is $M$-affiliated if and only if $W^{*}(T)\subset\LII(M)$.
\end{prp}
\begin{proof}
Note $W^{*}(T)=W^{*}\lc{}T,\mathcal{L}_{1_{M}}\rc$ since $\mathcal{L}$ is unital. Proposition \ref{PRP.Wstar_Generated} and $1)$ in Proposition \ref{PRP.FC_Bd} imply spectral projections in $T$ generate $W^{*}(T)$. Apply Lemma \ref{LEM.FC_Unitary_Com}.
\end{proof}

\begin{rem}\label{REM.Wstar_SMO_I}
For all densely defined closed operators $T$ on a Hilbert space $H$, get $T^{*}T$ self-adjoint and set absolute value $\absv{1.15}{T}:=\sqrt{T^{*}T}$ \lc{}cf.~Theorem 5.1.9 in \cite{BK.Ped.1989.Analysis_Now}\rc{}. If $T$ is $M$-affiliated, then $E_{\absv{1.15}{T}}(Z)=\chi_{Z}\lc\absv{1.15}{T}\rc\in\LII(M)$ for all $Z\in\mathfrak{B}(\mathbb{R})$ by Proposition \ref{PRP.Wstar_SMO_I}.
\end{rem}

Following Notation \ref{NTN.Wstar_Trace}, let $\tau$ further denote the push-forward of $\tau:M_{+}\longrightarrow [0,\infty]$ along $\mathcal{L}$ to $\LII(M)$. The f.s.n.~trace $\tau:\LII(M)\longrightarrow [0,\infty]$ has definition domain $\mathcal{L}\lc\mathfrak{m}_{\tau}\rc$.

\begin{dfn}\label{DFN.Wstar_SMO_II}
We call an $M$-affiliated operator $T$ on $\HII(M,\tau)$ $\tau$-measurable, or just measurable if there exists $\lambda>0$ s.t.~

\begin{align}\label{EQ.DFN.Wstar_SMO_I_1}
\tau\lc{}E_{\absv{1.15}{T}}\lc{}[\lambda,\infty)\rc\rc{}<\infty.    
\end{align}
\end{dfn}

\begin{rem}
Corollary IX.2.9 in \cite{BK.Tak.2003.OpAlg_II} ensures Definition \ref{DFN.Wstar_SMO_II} is $\tau$-measurability as used in \cite{BK.Tak.2003.OpAlg_II}. Proposition \ref{PRP.Wstar_SMO_II} further breaks down $\tau$-measurability for self-adjoint unbounded operators on $\HII(M,\tau)$.
\end{rem}

\begin{prp}\label{PRP.Wstar_SMO_II}
If $T\in\UBII(\HII(M,\tau))_{h}$, then $T$ is $\tau$-measurable if and only if

\begin{itemize}
\item[1)] $E_{T}(Z)\in\LII(M)$ for all $Z\in\mathfrak{B}(\mathbb{R})$,

\item[2)] there exists $\lambda>0$ s.t.~$E_{T}\lc{}(-\infty,-\lambda]\rc{},E_{T}\lc{}[\lambda,\infty)\rc\in\mathcal{L}\lc\mathfrak{m}_{\tau}\rc$.
\end{itemize}
\end{prp}
\begin{proof}
Proposition \ref{PRP.Wstar_SMO_I} at once implies $1)$ is equivalent to $T$ being $M$-affiliated. For all $\lambda>0$, get $\chi_{[\lambda,\infty)}(\absv{1}{t})=\chi_{(-\infty,-\lambda]}(t)+\chi_{[\lambda,\infty)}(t)$ for all $t\in\mathbb{R}$. Equation \ref{EQ.DFN.Wstar_SMO_I_1} therefore shows $2)$ is equivalent to $\tau$-measurability for all self-adjoint $M$-affiliated operators.
\end{proof}

Proposition \ref{PRP.Wstar_SMO_III} collects properties of $\mathcal{L}$. In particular, $3)$ states the maximality property. Using maximality and Remark \ref{REM.Wstar_SMO_II}, we readily see extending the $^{*}$-algebra structure of $M$ to $L^{0}(M,\tau)$ yields a $^{*}$-algebra. Note closure is necessary for this.

\begin{prp}\label{PRP.Wstar_SMO_III}
We know each $\mathcal{L}_{x}$ is a $\tau$-measurable operator on $\HII(M,\tau)$ for all $x\in L^{0}(M,\tau)$ and furthermore have the following.

\begin{itemize}
\item[1)] For all $x,y\in L^{0}(M,\tau)$ and $\lambda\in\mathbb{C}$, we have

\begin{itemize}
\item[1.1)] $\mathcal{L}_{\lambda_{1}x+\lambda_{2}y}=\overline{\lambda_{1}\mathcal{L}_{x}+\lambda_{2}\mathcal{L}_{y}}$,

\item[1.2)] $\mathcal{L}_{xy}=\overline{\mathcal{L}_{x}\mathcal{L}_{y}}$,

\item[1.3)] $\mathcal{L}_{x^{*}}=\mathcal{L}_{x}^{*}$.
\end{itemize}

\begin{reapply}
\end{reapply}

\item[2)] If $T$ is $\tau$-measurable, then there exists unique $x\in L^{0}(M,\tau)$ s.t.~$T=\mathcal{L}_{x}$.

\item[3)] If $x,y\in L^{0}(M,\tau)$ s.t.~$\mathcal{L}_{y}\subset\mathcal{L}_{x}$, then $\mathcal{L}_{x}=\mathcal{L}_{y}$.
\end{itemize}
\end{prp}
\begin{proof}
Apply Theorem IX.2.5 in \cite{BK.Tak.2003.OpAlg_II}.
\end{proof}


\pagebreak


\begin{rem}\label{REM.Wstar_SMO_II}
For all densely defined closable operators $T$ on a Hilbert space $H$, get $T^{*}$ densely defined closed and $T^{*}=\lc{}T^{*}\rc^{**}=\lc\overline{T}\rc^{*}$ by $\overline{T}=T^{**}$ \lc{}cf.~Theorem 5.15 in \cite{BK.Ped.1989.Analysis_Now}\rc{}.
\end{rem}

Partial order on $\UBII\lc\HII(M,\tau)\rc_{h}$ is fixed by Definition \ref{DFN.Unbd_PO}. We pull back partial order to $L^{0}(M,\tau)_{h}$ along $\mathcal{L}$. The set $L^{0}(M,\tau)_{h}$ of hermitian elements in Definition \ref{DFN.Wstar_SMO_III} below is given using algebra involution.

\begin{dfn}\label{DFN.Wstar_SMO_III}
The hermitian, resp.~positive elements in $L^{0}(M,\tau)$ are

\begin{align}\label{EQ.DFN.Wstar_SMO_III_1}
L^{0}(M,\tau)_{h}:=\lset{}x\in L^{0}(M,\tau)\ \vset\ \mathcal{L}_{x}=\mathcal{L}_{x}^{*}\rset{},\ L^{0}(M,\tau)_{+}:=\lset{}x\in L^{0}(M,\tau)_{h}\ \vset\ \mathcal{L}_{x}\geq 0\rset{}.
\end{align}
\end{dfn}

\begin{ntn}
Rather than hermitian, we say that $x\in L^{0}(M,\tau)_{h}$ is self-adjoint.
\end{ntn}

\begin{rem}
Corollary IX.2.10 in \cite{BK.Tak.2003.OpAlg_II} shows $L^{0}(M,\tau)_{+}$ is positive cone generating the partial order on $L^{0}(M,\tau)_{h}$. Proposition \ref{PRP.Wstar_NCI_V} implies the set $L^{0}(M,\tau)_{+}$ of positive elements generates the partial order on $L^{0}(M,\tau)$ as per Definition \ref{DFN.PO_I}.
\end{rem}

If an application of functional calculus preserves $\tau$-measurability, then we obtain a unique element in $L^{0}(M,\tau)$ by $2)$ in Proposition \ref{PRP.Wstar_SMO_III}. Taking absolute values preserves $\tau$-measurability. This lets us define generalised singular numbers by Equation \ref{EQ.SSEC.B_SMO_NCI_5}.

\begin{dfn}\label{DFN.Wstar_SMO_IV}
Let $x\in L^{0}(M,\tau)_{h}$. If $g\in\SII\lc{}E_{\mathcal{L}_{x}}\rc$ s.t.~$g\lc\mathcal{L}_{x}\rc$ is $\tau$-measurable, then let $g(x)\in L^{0}(M,\tau)$ be the unique element s.t.~$\mathcal{L}_{g(x)}=g\lc\mathcal{L}_{x}\rc$.
\end{dfn}

\begin{prp}\label{PRP.Wstar_SMO_IV}
Let $x\in L^{0}(M,\tau)$.

\begin{itemize}
\item[1)] If $x\in L^{0}(M,\tau)_{h}$ and $g\in L^{\infty}\lc\spec\mathcal{L}_{x},dE_{\mathcal{L}_{x}}\rc$, then $g\lc\mathcal{L}_{x}\rc$ is $\tau$-measurable.

\item[2)] If $x\in L^{0}(M,\tau)_{+}$ and $p\in [1,\infty)$, then $\mathcal{L}_{x}^{p}$ is $\tau$-measurable.

\item[3)] $\absv{1.15}{\mathcal{L}_{x}}$ is $\tau$-measurable and $\absv{1.15}{\mathcal{L}_{x}}=\mathcal{L}_{\sqrt{x^{*}x}}$.
\end{itemize}
\end{prp}
\begin{proof}
If $x\in L^{0}(M,\tau)_{h}$ and $g\in L^{\infty}\lc\spec\mathcal{L}_{x},dE_{\mathcal{L}_{x}}\rc$, then $g\lc\mathcal{L}_{x}\rc\in\LII(M)$ is $\tau$-measurable by Proposition \ref{PRP.Wstar_SMO_III}. The latter further implies $3)$ if $2)$ holds. For this, merely apply $2)$ using $p=2$. Get $2)$ since Equation \ref{EQ.DFN.Wstar_SMO_I_1} demands fix but arbitrarily large $\lambda\in (0,\infty)$ while $\lambda^{-p}\uparrow\infty$ as $\lambda\uparrow\infty$ for all $p\in [1,\infty)$.
\end{proof}

\begin{dfn}
For all $x\in L^{0}(M,\tau)$, set $\absv{0.925}{x}:=\sqrt{x^{*}x}$.
\end{dfn}

We extend $\tau$ to $L^{0}(M,\tau)_{+}$ \lc{}cf.~pp.1461-1470 in \cite{BK.Joh_Lin.2003.Wstar_Lp_II}\rc{}. The extension is linear. For all $x\in L^{0}(M,\tau)$, we have $E_{\mathcal{L}_{\absv{0.925}{x}}}\lc{}[\lambda,\infty)\rc\in\LII(M)$ as per Remark \ref{REM.Wstar_SMO_I}. For all $x\in L^{0}(M,\tau)$, we define the generalised singular number $\mu(x):(0,\infty)\longrightarrow [0,\infty)$ of $x$ by setting

\begin{align}\label{EQ.SSEC.B_SMO_NCI_5}
\mu_{t}(x):=\inf\hspace{0.025cm} \lset\lambda>0\ \vset\ \tau\lc{}E_{\mathcal{L}_{\absv{0.925}{x}}}\lc{}[\lambda,\infty)\rc\rc\leq t\rset{}
\end{align}

\noindent for all $t>0$.

\begin{dfn}\label{DFN.Wstar_SMO_Trace}
For all $x\in L^{0}(M,\tau)_{+}$, the trace of $x$ is defined by

\begin{align}\label{EQ.DFN.Wstar_SMO_Trace_1}
\tau(x):=\int_{0}^{\infty}\mu_{t}(x)dt.
\end{align}
\end{dfn}

\begin{rem}\label{REM.Wstar_SMO_Trace}
For all $x\in L^{0}(M,\tau)$ and $t>0$, note Equation \ref{EQ.SSEC.B_SMO_NCI_5} immediately shows we have $\mu_{t}(x)=\mu_{t}(x^{*})=\mu_{t}\lc\absv{0.925}{x}\rc$ by definition.
\end{rem}


\subsubsection*{Noncommutative $L^{p}$-spaces and integration}

Extension of integration theory to the noncommutative setting is fundamental to our discussion. Proposition \ref{PRP.Wstar_Trace_Ext_I} and Proposition \ref{PRP.Wstar_Trace_Ext_III} reduce the case of tracial $C^{*}$-~to tracial $W^{*}$-algebras.\par
Let $(M,\tau)$ be a tracial $W^{*}$-algebra. For all $p\in [1,\infty)$, $2)$ in Proposition \ref{PRP.Wstar_SMO_IV} and Equation \ref{EQ.DFN.Wstar_SMO_Trace_1} let us define noncommutative $L^{p}$-spaces. For $p=\infty$, we use $M$.

\begin{dfn}\label{DFN.NC_Int_Lp}
Let $p\in [1,\infty]$.

\begin{itemize}
\item[1)] Assume $p<\infty$. The noncommutative $L^{p}$-space of $(M,\tau)$ is

\begin{align}\label{EQ.DFN.NC_Int_Lp_1}
L^{p}(M,\tau):=\lset{}x\in L^{0}(M,\tau)\ \vset\ \tau\lc\absv{0.925}{x}^{p}\rc^{\frac{1}{p}}<\infty\rset{}.
\end{align}

\begin{reapply}
\end{reapply}

\noindent For all $x\in L^{p}(M,\tau)$, set $\|x\|_{p}:=\tau\lc\absv{0.925}{x}^{p}\rc^{\frac{1}{p}}$. We further call $\|.\|_{p}$ the noncommutative $L^{p}$-norm. The self-adjoint, resp.~positive elements in $L^{p}(M,\tau)$ are

\begin{align}\label{EQ.DFN.NC_Int_Lp_2}
L^{p}(M,\tau)_{h}:=L^{p}(M,\tau)\cap L^{0}(M,\tau)_{h},\ L^{0}(M,\tau)_{+}:=L^{p}(M,\tau)\cap L^{0}(M,\tau)_{+}.
\end{align}

\begin{reapply}
\end{reapply}

\item[2)] The noncommutative $L^{\infty}$-space of $(M,\tau)$ is $L^{\infty}(M,\tau):=M$.
\end{itemize}
\end{dfn}

\begin{prp}\label{PRP.Wstar_NCI_I}
For all $p\in [1,\infty]$, we have

\begin{itemize}
\item[1)] $\lc{}L^{p}(M,\tau),\|.\|_{p}\rc$ is a Banach space s.t.~$M\cap L^{p}(M,\tau)\subset L^{p}(M,\tau)$ is $\|.\|_{p}$-dense, \phantom{\big)}

\item[2)] $\Adj:L^{p}(M,\tau)\longrightarrow L^{p}(M,\tau)$ is anti-linear isometric involution, \phantom{\big)}

\item[3)] $\tau\in L^{1}(M,\tau)^{*}$ s.t.~$\tau(x)\geq 0$ for all $x\in L^{1}(M,\tau)_{+}$. \phantom{\big)}
\end{itemize} 
\end{prp}
\begin{proof}
All claims are given by \lc{}i\rc{} and \lc{}ii\rc{} in Theorem IX.2.13 in \cite{BK.Tak.2003.OpAlg_II}. Of course, its proof shows $\tau$ has linear extension to $L^{1}(M,\tau)$. We therefore have $\tau\in L^{1}(M,\tau)^{*}$ as claimed.
\end{proof}

As $\mathfrak{n}_{\tau}=M\cap L^{2}(M,\tau)\subset L^{2}(M,\tau)$ is $\|.\|_{2}$-dense by $1)$ in Proposition \ref{PRP.Wstar_NCI_I}, the identity $\id:\lc\mathfrak{n}_{\tau},\|.\|_{\tau}\rc\longrightarrow \lc\mathfrak{n}_{\tau},\|.\|_{2}\rc$ closes to an isometric isomorphism $\id_{\tau}:\HII(M,\tau)\longrightarrow L^{2}(M,\tau)$. Equivalence classes w.r.t.~$\|.\|_{\tau}$ are mapped to equivalence classes in uniform closure which are represented by square integrable $\tau$-measurable operators.

\begin{dfn}
We call $\id_{\tau}:\HII(M,\tau)\longrightarrow L^{2}(M,\tau)$ identity in measure topology.
\end{dfn}

\begin{ntn}
Let $\id_{\tau,\op}:\lc\HII(M,\tau),\|.\|_{\tau}\rc\longrightarrow L^{2}\lc{}M^{\op},\tau\rc$ denote identity in measure topology using $\lc{}M^{\op},\tau\rc$ instead. We may consider it by $2)$ in Proposition \ref{PRP.Wstar_Trace_CRA_I}.
\end{ntn}


\pagebreak


We have $\id_{\tau}^{-\dagger}=\lc\id_{\tau}^{-1}\rc^{\dagger}$ as per Definition \ref{DFN.Unbd_Twist}.

\begin{prp}\label{PRP.Wstar_NCI_II}
The maps $\id_{\tau}$ and $\id_{\tau}^{-1}$ are continuous w.r.t.~measure topology on $L^{0}(M,\tau)$ and $\HII(M,\tau)$. For all $x\in L^{0}(M,\tau)$, we have

\begin{itemize}
\item[1)] $\dom\id_{\tau}^{-\dagger}\lc\mathcal{L}_{x}\rc{}=\big\{\hspace{0.025cm} u\in L^{2}(M,\tau)\ \vset\ xu\in L^{2}(M,\tau)\hspace{0.025cm} \big\}$, \phantom{\big)}

\item[2)] $\id_{\tau}^{-\dagger}\lc\mathcal{L}_{x}\rc{}(u)=xu$ in $L^{0}(M,\tau)$ for all $u\in\dom\id_{\tau}^{-\dagger}\lc\mathcal{L}_{x}\rc$. \phantom{\big)}
\end{itemize}
\end{prp}
\begin{proof}
We directly verify continuity of $\id_{\tau}$ and $\id_{\tau}^{-1}$ in measure topologies on uniform structures given by Equation \ref{EQ.SSEC.B_SMO_NCI_1} and Equation \ref{EQ.SSEC.B_SMO_NCI_2}. For all $x\in M$, our claims follow since they reduce to canonical left-action of $M$ on $\HII(M,\tau)$. Construction of $\mathcal{L}$ therefore implies the general case since $\id_{\tau}$ and $\id_{\tau}^{-1}$ are continuous in measure topologies.
\end{proof}

\begin{prp}\label{PRP.Wstar_NCI_III}
We have $x\in L^{1}(M,\tau)_{+}$ if and only if $\sqrt{x}\in L^{2}(M,\tau)_{+}$.
\end{prp}
\begin{proof}
By definition of noncommutative $L^{1}$-, resp.~$L^{2}$-spaces.
\end{proof}

\begin{prp}\label{PRP.Wstar_NCI_IV}
For all $x\in L^{0}(M,\tau)$ and $p\in [1,\infty]$, we have

\begin{itemize}
\item[1)] $x=\RE(x)+i\IM(x)$ and $\RE(x):=\frac{x+x^{*}}{2},\IM(x):=-i\frac{x-x^{*}}{2}\in L^{0}(M,\tau)_{h}$,

\item[2)] $x=x_{+}-x_{-}$ for $x_{+}:=\max\{x,0\},x_{-}:=-\min\{x,0\}\in L^{0}(M,\tau)_{+}$ if $x\in L^{0}(M,\tau)_{h}$,

\item[3)] $x\in L^{p}(M,\tau)$ if and only if $\RE(x)_{+},\RE(x)_{-},\IM(x)_{+},\IM(x)_{-}\in L^{p}(M,\tau)$.
\end{itemize}
\end{prp}
\begin{proof}
Get $1)$ by $1)$ in Proposition \ref{PRP.Wstar_SMO_III}. Get $2)$ by Proposition \ref{PRP.Wstar_SMO_II} together with Proposition \ref{PRP.Wstar_SMO_III}. We see $3)$ follows from $2)$ since $\absv{0.925}{x}^{p}=\lc{}x_{+}\rc^{p}+\lc{}x_{-}\rc^{p}$ in each case.
\end{proof}

\begin{rem}
$\RE$ and $\IM$ are $\mathbb{R}$-linear maps on $L^{0}(M,\tau)$. For all $x\in L^{0}(M,\tau)$, we have $x^{*}=\RE(x)-i\IM(x)$ by anti-linearity of taking adjoints.
\end{rem}

\begin{prp}\label{PRP.Wstar_NCI_V}
$L^{0}(M,\tau)_{+}$ generates the partial order on $L^{0}(M,\tau)$.
\end{prp}
\begin{proof}
We use $L^{0}(M,\tau)_{h}$ as hermitian elements. Definition \ref{DFN.Unbd_PO} fixes partial order. Corollary IX.2.10 in \cite{BK.Tak.2003.OpAlg_II} and $2)$ in Proposition \ref{PRP.Wstar_NCI_IV} show $L^{0}(M,\tau)_{+}$ is a proper cone generating the partial order on $L^{0}(M,\tau)_{h}$.
\end{proof}


\subsubsection*{The modified standard pairing}

Let $(M,\tau)$ be a tracial $W^{*}$-algebra. We know $\tau\in L^{1}(M,\tau)_{+}^{*}$ by $3)$ in Proposition \ref{PRP.Wstar_NCI_I}. Equation \ref{EQ.SSEC.B_SMO_NCI_6} are H\"older inequalities. These in turn yield a modified standard pairing defined by Equation \ref{EQ.SSEC.B_SMO_NCI_7}.\par
Let $p,q\in [1,\infty]$ s.t.~$1=p^{-1}+q^{-1}$. For all $x\in L^{p}(M,\tau)$ and $y\in L^{q}(M,\tau)$, we apply \lc{}iv\rc{} in Theorem IX.2.13 in \cite{BK.Tak.2003.OpAlg_II} to get $xy\in L^{1}(M,\tau)$ and

\begin{align}\label{EQ.SSEC.B_SMO_NCI_6}
\absv{1.15}{\tau\lc{}xy\rc{}}\leq \|x\|_{p}\| y\|_{q}.
\end{align}

\noindent For $p=\infty$, we use $\|.\|_{M}$. By \lc{}iv\rc{} in Theorem IX.2.13 in \cite{BK.Tak.2003.OpAlg_II}, note Equation \ref{EQ.SSEC.B_SMO_NCI_6} defines bounded non-degenerate pairing $S:L^{p}(M,\tau)\times L^{q}(M,\tau)\longrightarrow\mathbb{C}$ by setting $S(x,y):=\tau\lc{}xy\rc$ for all $x\in L^{p}(M,\tau)$ and $y\in L^{q}(M,\tau)$. We call $S$ the standard pairing.\par


\pagebreak


In order to recover the GNS-inner product of $\tau$ for $p=q=2$ as per Equation \ref{EQ.SSEC.B_SMO_Wstar_Trace_1}, we modify the fist variable by taking adjoints. We therefore define the modified standard pairing by setting

\begin{align}\label{EQ.SSEC.B_SMO_NCI_7}
x^{\flat}(y):=\tau(x^{*}y)
\end{align}

\noindent for all $x\in L^{p}(M,\tau)$ and $y\in L^{q}(M,\tau)$. We have $\tau\lc{}xy\rc{}=\tau\lc{}yx\rc$ and $\overline{\tau\lc{}xy\rc{}}=\tau\lc{}x^{*}y^{*}\rc$ in each case. For $p=q=2$, get $\tau\lc{}xy\rc{}=\lgl x^{*},y\rgl_{\tau}$ for all $x,y\in L^{2}(M,\tau)$. The modified standard pairing is bounded and non-degenerate. The standard and modified standard pairing are identical upon restriction to self-adjoint elements in the first variable.

\begin{dfn}\label{DFN.Wstar_NCI_MSP}
For all $p,q\in [1,\infty]$ s.t.~$1=p^{-1}+q^{-1}$, the modified standard pairing on $L^{p}(M,\tau)\times L^{q}(M,\tau)$ is defined by Equation \ref{EQ.SSEC.B_SMO_NCI_7}.
\end{dfn}

\begin{prp}\label{PRP.Wstar_NCI_VI}
Let $M_{*}$ be the pre-dual of $M$.

\begin{itemize}
\item[1)] $\flat:L^{1}(M,\tau)\longrightarrow M^{*}$ is an anti-linear isometry onto $M_{*}\subset M^{*}$.

\item[2)] $L^{1}(M,\tau)^{\flat}$ is the set of all normal bounded functionals on $M$.

\item[3)] For all $x\in L^{1}(M,\tau)$, we have

\begin{itemize}
\item[3.1)] $x$ is self-adjoint if and only if $x^{\flat}$ is real,

\item[3.2)] $x$ is positive if and only if $x^{\flat}$ is positive.
\end{itemize}

\begin{reapply}
\end{reapply}

\end{itemize}
\end{prp}
\begin{proof}
Get $1)$ by \lc{}iv\rc{} in Theorem IX.2.13 in \cite{BK.Tak.2003.OpAlg_II}. Using $1)$, get $2)$ by Corollary III.3.11 in \cite{BK.Tak.1979.OpAlg_I}. We directly verify $3)$ using Proposition \ref{PRP.Wstar_NCI_III} and Proposition \ref{PRP.Wstar_NCI_IV}.
\end{proof}

\begin{rem}\label{REM.Wstar_NCI_MSP}
Following Proposition \ref{PRP.Wstar_NCI_VI}, set $M_{*}:=L^{1}(M,\tau)$. We readily see $3)$ in the proposition shows the partial order induced by $L^{1}(M,\tau)\subset L^{0}(M,\tau)$ equals the dual space partial order induced by $M_{*}\subset M^{*}$ as per $2)$ in Proposition \ref{PRP.Cstar_PO}.
\end{rem}

\begin{dfn}\label{DFN.Wstar_NCI_State_Space}
We define the state space $\SII(M):=\lset\mu\in M_{+}^{*}\ \vset\ \|\mu\|_{M}=1\rset$ and the normal state space $\mathcal{S}^{\NI}(M):=\SII(M)\cap L^{1}(M,\tau)^{\flat}$ of $M$.
\end{dfn}

\begin{prp}\label{PRP.Wstar_NCI_VIII}
Let $A\subset M$ be a strongly dense $C^{*}$-subalgebra. If $\AII\subset A\cap L^{2}(M,\tau)$ is $\|.\|_{A}$-dense in $A$ and $\|.\|_{2}$-dense in $L^{2}(M,\tau)$, then

\begin{itemize}
\item[1)] $\AII\subset M$ strongly dense,

\item[2)] finite convex combinations of $\AII^{*}\cdot \AII\subset L^{1}(M,\tau)$ are $\|.\|_{1}$-dense in $L^{1}(M,\tau)$.
\end{itemize}
\end{prp}
\begin{proof}
Get $1)$ as $A\subset M$ is strongly dense and $\AII\subset A$ is $\|.\|_{A}$-dense. Mazur's lemma \cite{BK.Tak.1979.OpAlg_I} implies $w^{*}$-density of $\AII^{*}\cdot \AII\subset L^{1}(M,\tau)_{+}$ suffices for $\|.\|_{1}$-density. Let $x\in L^{1}(M,\tau)_{+}$ and $\{\hspace{0.0125cm} y_{n}\hspace{0.0025cm}\}_{n\in\mathbb{N}}\subset\AII$ s.t.~$\sqrt{x}=\|.\|_{2}$-$\lim_{n\in\mathbb{N}}y_{n}$. For all $n\in\mathbb{N}$, get $y_{n}^{*}y_{n}\in L^{1}(M,\tau)_{+}$. For all $z\in M$, we see $\tau\lc{}xz\rc{}=\lgl\sqrt{x},\sqrt{x}z\rgl_{2}=\lim_{n\in\mathbb{N}}\lgl y_{n},y_{n}z\rgl_{2}=\lim_{n\in\mathbb{N}}\tau\lc{}y_{n}^{*}y_{n}z\rc$. This is $w^{*}$-density.
\end{proof}


\subsection{Canonical left-~and right-actions of measurable operators}\label{SSEC.B_SMO_CLRA}

We extend canonical left-~and right-actions to spaces of measurable operators. We keep natural identifications as per Remark \ref{REM.Wstar_SMO_CLA} and Remark \ref{REM.Wstar_CLRA} explicit to ensure their consistent use. This subsumes the bounded case. Using canonical left-~and right-actions accordingly, we define spectral and joint spectral measures of self-adjoint measurable operators. This lets us formulate their bounded measurable joint functional calculus.


\subsubsection*{Definition using $^{*}$-algebra multiplication}

Let $(M,\tau)$ be a tracial $W^{*}$-algebra. Following Proposition \ref{PRP.Wstar_NCI_II}, note Definition \ref{DFN.Wstar_CLRA} gives canonical left-~and canonical right-action of $L^{0}(M,\tau)$ on $L^{2}(M,\tau)$ using multiplication in $L^{0}(M,\tau)$.

\begin{dfn}\label{DFN.Wstar_CLRA}
Let $x\in L^{0}(M,\tau)$. Set 

\begin{align}\label{EQ.DFN.Wstar_CLRA_1_2}
\dom L_{x,M} &:= \lset{}u\in L^{2}(M,\tau)\ \vset\ xu\in L^{2}(M,\tau)\rset{}, \\
\bigskip
\bigskip
\dom R_{x,M} &:= \lset{}u\in L^{2}(M,\tau)\ \vset\ ux\in L^{2}(M,\tau)\rset{}.
\end{align}

\noindent We define canonical left-action $L_{x,M}$ of $x$ on $M$, resp.~canonical right-action $R_{x,M}$ of $x$ on $M$ by setting

\begin{align}\label{EQ.DFN.Wstar_CLRA_3}
L_{x,M}(u):=xu,\ R_{x,M}(u):=ux
\end{align}

\noindent for all $u\in\dom L_{x,M}$, resp.~for all $u\in\dom R_{x,M}$.
\end{dfn}

We equip the opposite algebra $L^{0}(M,\tau)^{\op}$ of $L^{0}(M,\tau)$ as per Definition \ref{DFN.Oppalg} with the measure topology of $L^{0}(M,\tau)$. Using Corollary \ref{COR.Wstar_CLRA_III}, we readily see Definition \ref{DFN.Wstar_CLRA} determines, by Equation \ref{EQ.DFN.Wstar_CLRA_3}, two unbounded faithful unital $^{*}$-representations

\begin{align}\label{EQ.SSEC.B_SMO_CLRA_1}
L_{M}:L^{0}(M,\tau)\longrightarrow\UBII\lc{}L^{2}(M,\tau)\rc{},\ R_{M}:L^{0}(M,\tau)^{\op}\longrightarrow\UBII\lc{}L^{2}(M,\tau)\rc{}.
\end{align}

\begin{dfn}\label{DFN.Wstar_CLRA_Unbd_Representation}
We call $L_{M}$ and $R_{M}$ in Equation \ref{EQ.SSEC.B_SMO_CLRA_1} canonical left-~, resp.~canonical right-action of $L^{0}(M,\tau)$ on $L^{2}(M,\tau)$.
\end{dfn}

\begin{ntn}\label{NTN.Wstar_CLRA}
Unless stated otherwise, we suppress $W^{*}$-algebras in subscripts of canonical left-~and right-actions. We require subscripts in Section \ref{SEC.B_JFC}. We further write $L^{\op}$, suppressing subscripts, for the canonical left-action of $L^{0}\lc{}M^{\op},\tau\rc$ on $L^{2}(M,\tau)$.
\end{ntn}

Proposition \ref{PRP.Wstar_CLRA_I} states $L$ is $\mathcal{L}$ up to twisting with $\id_{\tau}^{-1}$. We see $L$ subsumes our discussion for $\mathcal{L}$ in Subsection \ref{SSEC.B_SMO_NCI}, in particular the bounded case. Concerning results for $R$, get $R\cong L^{\op}$ naturally but $R\neq L^{\op}$ in general. Following Remark \ref{REM.Wstar_CLRA}, we know results for $L^{\op}$ apply to $R$ if and only if they are preserved under $R\cong L^{\op}$.

\begin{prp}\label{PRP.Wstar_CLRA_I}
For all $x\in L^{0}(M,\tau)$, we have $L_{x}=\id_{\tau}^{-\dagger}\lc\mathcal{L}_{x}\rc$.
\end{prp}
\begin{proof}
Apply Proposition \ref{PRP.Wstar_NCI_II}.
\end{proof}


\pagebreak


Lemma \ref{LEM.Wstar_CLRA_II} shows we obtain $R\cong L^{\op}$ by twisting with natural $^{*}$-isomorphism $L^{0}(M,\tau)^{\op}\cong L^{0}\lc{}M^{\op},\tau\rc$ extending $\id_{M^{\op}}$. This natural $^{*}$-isomorphism is called opposite algebra map and constructed in Lemma \ref{LEM.Wstar_CLRA_I}.

\begin{lem}\label{LEM.Wstar_CLRA_I}
There exists unique $^{*}$-isomorphism $\op:L^{0}(M,\tau)^{\op}\longrightarrow L^{0}\lc{}M^{\op},\tau\rc$ s.t.~

\begin{itemize}
\item[1)] $\op$ is a homeomorphism w.r.t.~measure topologies on $L^{0}(M,\tau)^{\op}$ and $L^{0}\lc{}M^{\op},\tau\rc$,

\item[2)] $\op^{-1}:L^{0}\lc{}M^{\op},\tau\rc\longrightarrow L^{0}(M,\tau)^{\op}$ is a $^{*}$-isomorphism,

\item[3)] we have commutative diagram

\begin{equation}\label{EQ.LEM.Wstar_CLRA_I_1}
\begin{tikzcd}
M^{\op}\arrow[rr,hook]\arrow[dd,"\id_{M}"] & & L^{0}(M,\tau)^{\op}\arrow[dd,"\op"] \\
& & \\
M^{\op}\arrow[rr,hook] & & L^{0}\lc{}M^{\op},\tau\rc{}
\end{tikzcd}
\end{equation}

\begin{reapply}
\end{reapply}

\noindent of injective horizontal and bijective vertical maps.
\end{itemize}
\end{lem}
\begin{proof}
If $\{x_{k}\}_{k\in N}\subset M$ is a net, then Equation \ref{EQ.SSEC.B_SMO_NCI_1} and continuity of $\Adj$ in measure topology shows $\{x_{k}\}_{k\in K}\subset M$ is Cauchy in measure if and only if $\{x_{k}\}_{k\in K}\subset M^{\op}$ is. For all $x=\vstretch{0.8875}{\big[}\{x_{k}\}_{k\in N}\vstretch{0.8875}{\big]}\in L^{0}(M,\tau)$, let $x^{\op}:=\vstretch{0.8875}{\big[}\{x_{k}\}_{k\in N}\vstretch{0.8875}{\big]}\in L^{0}\lc{}M^{\op},\tau\rc$ and set

\begin{align}\label{EQ.LEM.Wstar_CLRA_I_2}
\op(x):=x^{\op}.
\end{align}

\noindent By construction, $\op:L^{0}(M,\tau)^{\op}\longrightarrow L^{0}\lc{}M^{\op},\tau\rc$ satisfies Diagram \ref{EQ.LEM.Wstar_CLRA_I_1} and is continuous in measure topology. Equation \ref{EQ.LEM.Wstar_CLRA_I_2} in turn is fully determined by Diagram \ref{EQ.LEM.Wstar_CLRA_I_1} and continuity in measure topology. We see $\op$ is unique. It is a homeomorphism since $\op^{-1}$ is determined by mapping Cauchy nets in $M^{\op}$ to Cauchy nets in $M$. We are left to show $\op$, ergo $\op^{-1}$, is a $^{*}$-homomorphism. The $^{*}$-algebras $M$ and $M^{\op}$ extend suitably.
\end{proof}

\begin{dfn}
We call $\op:L^{0}(M,\tau)^{\op}\longrightarrow L^{0}\lc{}M^{\op},\tau\rc$ defined by Equation \ref{EQ.LEM.Wstar_CLRA_I_2} the opposite algebra map.
\end{dfn}

\begin{cor}\label{COR.Wstar_CLRA_I}
We have commutative diagram

\smallskip

\begin{equation}\label{EQ.COR.Wstar_CLRA_I_1}
\begin{tikzcd}
\lc\HII(M,\tau),\|.\|_{\tau}\rc\arrow[rr,"\id_{\tau}"]\arrow[dd,"\id_{\HII(M,\tau)}"] & & \lc{}L^{2}(M,\tau),\|.\|_{2}\rc\arrow[dd,"^{\op}"] \\
& & \\
\lc\HII(M,\tau),\|.\|_{\tau}\rc\arrow[rr,"\id_{\tau,\op}"] & & \lc{}L^{2}\lc{}M^{\op},\tau\rc{},\|.\|_{2}\rc{}
\end{tikzcd}
\end{equation}

\medskip

\noindent of isometric isomorphisms of Hilbert spaces.
\end{cor}
\begin{proof}
Apply the construction of $\op$ in the proof of Lemma \ref{LEM.Wstar_CLRA_I}.
\end{proof}


\pagebreak


Note $2)$ in Proposition \ref{PRP.Wstar_NCI_I} shows $\Adj:L^{2}(M,\tau)\longrightarrow L^{2}(M,\tau)$ defines $\Adj^{\dagger}$ as per Definition \ref{DFN.Unbd_Twist}. Corollary \ref{COR.Wstar_CLRA_I} shows $\op:L^{2}(M,\tau)\longrightarrow L^{2}\lc{}M^{\op},\tau\rc$ defines $\op^{\dagger}$. If we restrict to the bounded case, then we have commutative diagram

\smallskip

\begin{equation}\label{EQ.SSEC.B_SMO_CLRA_2}
\begin{tikzcd}
M\arrow[rr,"L"]\arrow[dd,"\Adj"] & & \BII\lc{}L^{2}(M,\tau)\rc\arrow[dd,"\Adj^{\dagger}"] \\
& & \\
M^{\op}\arrow[rr,"R"]\arrow[dd,"\id_{M}"] & & \BII\lc{}L^{2}(M,\tau)\rc\arrow[dd,"\op^{\dagger}"] \\
& & \\
M^{\op}\arrow[rr,"L^{\op}"] & & \BII\lc{}L^{2}\lc{}M^{\op},\tau\rc\rc{}
\end{tikzcd}
\end{equation}

\medskip

\noindent by Diagram \ref{EQ.PRP.Wstar_Trace_CLRA_Bd_1} and Diagram \ref{EQ.LEM.Wstar_CLRA_I_1}. We thus recover $R=L^{\op}$, given as $\mathcal{R}=\mathcal{L}^{\op}$ in $1)$ in Proposition \ref{PRP.Wstar_Trace_CRA_II}, if we collapse the lower part of Diagram \ref{EQ.SSEC.B_SMO_CLRA_2} by pull-back along Diagram \ref{EQ.COR.Wstar_CLRA_I_1}. We further account for twisting of $\mathcal{L}^{\op}$ with $\id_{\tau,\op}$.

\begin{lem}\label{LEM.Wstar_CLRA_II}
We have commutative diagram

\smallskip

\begin{equation}\label{EQ.LEM.Wstar_CLRA_II_1}
\begin{tikzcd}
L^{0}(M,\tau)\arrow[rr,"L"]\arrow[dd,"\Adj"] & & \UBII\lc{}L^{2}(M,\tau)\rc\arrow[dd,"\Adj^{\dagger}"] \\
& & \\
L^{0}(M,\tau)^{\op}\arrow[rr,"R"]\arrow[dd,"\op"] & & \UBII\lc{}L^{2}(M,\tau)\rc\arrow[dd,"\op^{\dagger}"] \\
& & \\
L^{0}\lc{}M^{\op},\tau\rc\arrow[rr,"L^{\op}"] & & \UBII\lc{}L^{2}\lc{}M^{\op},\tau\rc\rc{}
\end{tikzcd}
\end{equation}

\medskip

\noindent of injective horizontal and bijective vertical maps.
\end{lem}
\begin{proof}
For all $x\in L^{0}(M,\tau)$, we directly verify $\dom R_{x}=\dom\Adj^{\dagger}\lc{}L_{x^{*}}\rc$ and $R=\Adj^{\dagger}\circ\hspace{0.075cm} L$. This is the upper diagram. Lemma \ref{LEM.Wstar_CLRA_I} and $1)$ in Proposition \ref{PRP.Unbd_Twist} ensure vertical maps are bijective. In particular, taking adjoints is. Thus $L$, $R=\Adj^{\dagger}\circ\hspace{0.075cm} L$ and $L^{\op}$ are injective, hence we are left to show the lower diagram.\par
Let $x\in L^{0}(M,\tau)$. Note we have $\dom R_{x}=\lset{}u\in L^{2}(M,\tau)\ \vset\ ux\in L^{2}(M,\tau)\rset$ and $\dom L_{x^{\op}}^{\op}=\lset{}v\in L^{2}\lc{}M^{\op},\tau\rc\ \vset\ x^{\op}v\in L^{2}\lc{}M^{\op},\tau\rc\rset$ by definition. For all $u\in L^{2}(M,\tau)$, we calculate

\begin{align}\label{EQ.LEM.Wstar_CLRA_II_2}
R_{x}(u) = x\cdot^{\op}u & = \textrm{op}^{-1}\lc{}x^{\op}\rc\cdot^{\op}\textrm{op}^{-1}\lc{}u^{\op}\rc{}=\textrm{op}^{-1}\lc{}L_{x^{\op}}^{\op}\lc\op(u)\rc\rc{}.
\end{align}

\noindent Equation \ref{EQ.LEM.Wstar_CLRA_II_2} implies $\op\lc\dom R_{x}\rc{}=\dom L_{x^{\op}}^{\op}$. Thus $R_{x}=\op^{-\dagger}\lc{}L_{x^{\op}}^{\op}\rc$, hence $\op^{\dagger}(R_{x}){}=L_{x^{\op}}^{\op}$ upon applying the given dagger map. This is the lower diagram.
\end{proof}

\begin{cor}\label{COR.Wstar_CLRA_II}
For all $x\in L^{0}(M,\tau)$, we have

\begin{itemize}
\item[1)] $\Adj^{\dagger} L_{x}=R_{x^{*}}$ 

\item[2)] $R_{x}=\op^{-\dagger}\lc{}L_{x^{\op}}^{\op}\rc$.
\end{itemize}
\end{cor}
\begin{proof}
This reformulates the upper, resp.~lower part of Diagram \ref{EQ.LEM.Wstar_CLRA_II_1}.
\end{proof}

\begin{cor}\label{COR.Wstar_CLRA_III}
For all $x\in L^{0}(M,\tau)$, $L_{x}$ and $R_{x}$ are densely defined closed operators on $L^{2}(M,\tau)$. For all $x,y\in L^{0}(M,\tau)$ and $\lambda\in\mathbb{C}$, we have

\begin{itemize}
\item[1)] $L_{\lambda_{1}x+\lambda_{2}y}=\overline{\lambda_{1}L_{x}+\lambda_{2}L_{y}}$ and $R_{\lambda_{1}x+\lambda_{2}y}=\overline{\lambda_{1}R_{x}+\lambda_{2}R_{y}}$,

\item[2)] $L_{xy}=\overline{L_{x}L_{y}}$ and $R_{xy}=\overline{R_{y}R_{x}}$,

\item[3)] $L_{x^{*}}=L_{x}^{*}$ and $R_{x^{*}}=R_{x}^{*}$.
\end{itemize}
\end{cor}
\begin{proof}
Proposition \ref{PRP.Wstar_SMO_III} shows analogous claims for $\mathcal{L}$, and Proposition \ref{PRP.Unbd_Twist} implies they are preserved under twisting with $\id_{\tau}$. As such, Proposition \ref{PRP.Wstar_CLRA_I} shows all claims for $L$ at once. Proposition \ref{PRP.Unbd_Twist} implies they are preserved under twisting with $\op$. We therefore obtain all claims for $R$ by reducing to $L^{\op}$ using $2)$ in Corollary \ref{COR.Wstar_CLRA_II}.
\end{proof}

\begin{rem}\label{REM.Wstar_CLRA}
All canonical left-~and right-actions are multiplication by measurable operators. Corollary \ref{COR.Wstar_CLRA_III} shows they are unbounded faithful unital $^{*}$-representations extending the bounded case. This requires Proposition \ref{PRP.Wstar_CLRA_I}, i.e.~Proposition \ref{PRP.Wstar_NCI_II}. We twist with $\id_{\tau}$ in case of $L$, as well as with $\id_{\tau,\op}$ in case of $L^{\op}$. These identities in measure topology induce distinct measure topologies on $\HII(M,\tau)=\HII\lc{}M^{\op},\tau\rc$.\par
We obtain $L^{2}(M,\tau)$ and $L^{2}\lc{}M^{\op},\tau\rc$ accordingly. Thus $R\neq L^{\op}$ by Lemma \ref{LEM.Wstar_CLRA_II}, even as $R\cong L^{\op}$ is $R=L^{\op}$ upon pull-back along Diagram \ref{EQ.COR.Wstar_CLRA_I_1}. Note $R$ is not defined on an algebra of measurable operators, but on an opposite algebra of one. Since $R\cong L^{\op}$ up to twisting with opposite algebra maps, results for canonical left-actions apply if and only if they are compatible with such twisting. We use this in Corollary \ref{COR.Wstar_CLRA_III}, as well as in Section \ref{SEC.B_JFC}. Altogether, we have consistent use of canonical left-~and right-actions for joint functional calculus of self-adjoint measurable operators.
\end{rem}

\begin{lem}\label{LEM.Wstar_CLRA_III}
For all $x\in L^{0}(M,\tau)_{h}$, the following are equivalent:

\begin{itemize}
\item[1)] $u\in\dom L_{x}$,

\item[2)] $\RE(u)_{+},\RE(u)_{-},\IM(u)_{+},\IM(u)_{-}\in\dom L_{x}$.
\end{itemize} 
\end{lem}
\begin{proof}
Let $x\in L^{0}(M,\tau)_{h}$. For all $n\in\mathbb{N}$, set $x_{n}:=\chi_{[-n,n]}(x)x$. We know $\{x_{n}\}_{n\in\mathbb{N}}\in M_{h}$ by $1)$ in Proposition \ref{PRP.Wstar_SMO_IV}. Moreover, $\absv{0.925}{x_{n}}\leq \absv{0.925}{x_{n+1}}\leq \absv{0.925}{x}$ for all $n\in\mathbb{N}$. We know $u\in\dom L_{x}$ if and only if $\|xu\|_{2}=\int_{\spec L_{x}}\lambda^{2}dE_{L_{x}}^{u}<\infty$. Fatou's lemma implies 

\begin{align}\label{EQ.LEM.Wstar_CLRA_III_1}
\int_{\spec L_{x}}\lambda^{2}dE_{L_{x}}^{u}\leq\liminf_{n\in\mathbb{N}}\hspace{0.025cm} \int_{\spec L_{x}}\lc\lambda\cdot \chi_{[-n,n]}(\lambda)\rc^{2}dE_{L_{x}}^{u}=\|x_{n}u\|_{2}^{2}.
\end{align}

Since $\absv{0.925}{x_{n}}\leq \absv{0.925}{x}$ for all $n\in\mathbb{N}$, Equation \ref{EQ.LEM.Wstar_CLRA_III_1} implies

\begin{align}\label{EQ.LEM.Wstar_CLRA_III_2}
\|xu\|_{2}=\sup_{n\in\mathbb{N}}\hspace{0.025cm} \|x_{n}u\|_{2}=\lim_{n\in\mathbb{N}}\hspace{0.025cm} \|x_{n}u\|_{2}\in [0,\infty]
\end{align}

\noindent for all $u\in L^{2}(M,\tau)$. We use decomposition as per Proposition \ref{PRP.Wstar_NCI_IV} and apply limits in $n\in\mathbb{N}$ as per Equation \ref{EQ.LEM.Wstar_CLRA_III_2} to show equivalence as claimed.\par
Let $u\in L^{2}(M,\tau)$. Proposition \ref{PRP.Wstar_NCI_IV} implies 

\begin{align}\label{EQ.LEM.Wstar_CLRA_III_3}
\| zu\|_{2}^{2}=\dblv{}z\RE(u)\dblv_{2}^{2}+\dblv{}z\IM(u)\dblv_{2}^{2} 
\end{align}

\noindent for all $z\in M_{h}$. Note mixed terms $i2\RE\lgl z\RE(u),z\IM(u)\rgl_{2}$ do not appear in Equation \ref{EQ.LEM.Wstar_CLRA_III_3} as $\| zu\|_{2}\in\mathbb{R}$ ensures they vanish in each case. Since all positive and negative parts involved have disjoint support, multiplying out terms yields

\begin{align}\label{EQ.LEM.Wstar_CLRA_III_4}
\| zu\|_{2}^{2}=\dblv{}z\RE(u)_{+}\dblv_{2}^{2}+\dblv{}z\RE(u)_{-}\dblv_{2}^{2}+\dblv{}z\IM(u)_{+}\dblv_{2}^{2}+\dblv{}z\IM(u)_{-}\dblv_{2}^{2}
\end{align}

\noindent for all $z\in M_{h}$. In particular, Equation \ref{EQ.LEM.Wstar_CLRA_III_4} is satisfied using $z=x_{n}$ for all $n\in\mathbb{N}$. Thus applying the limit in $n\in\mathbb{N}$ as per Equation \ref{EQ.LEM.Wstar_CLRA_III_2} for given $u$ extends Equation \ref{EQ.LEM.Wstar_CLRA_III_4} to $z=x$ s.t.~the resulting limit is finite if and only if $u$ satisfies $1)$ and $2)$.
\end{proof}

\begin{cor}\label{COR.Wstar_CLRA_IV}
Let $x\in L^{1}(M,\tau)_{h}$.

\begin{itemize}
\item[1)] For all $n\in\mathbb{N}$, set $x_{n}:=\chi_{[-n,n]}(x)x$. Then $\{x_{n}\}_{n\in\mathbb{N}}\subset L^{1}(M,\tau)_{h}$ and $x=\|.\|_{1}$-$\lim_{n\in\mathbb{N}}x_{n}$.

\item[2)] Assume $x\in L^{1}(M,\tau)_{+}$. For all $n\in\mathbb{N}$, let $x_{n}:=\min\{x,n\}$. Then $\{x_{n}\}_{n\in\mathbb{N}}\subset L^{1}(M,\tau)_{+}$ and $x=\|.\|_{1}$-$\lim_{n\in\mathbb{N}}x_{n}$. We have $u\in\dom L_{x}$ if and only if $\sup_{n\in\mathbb{N}}\|x_{n}u\|_{2}<\infty$ or $\sup_{n\in\mathbb{N}}\|ux_{n}\|_{2}<\infty$.
\end{itemize}
\end{cor}
\begin{proof}
Arguing as for Equation \ref{EQ.LEM.Wstar_CLRA_III_2} in the proof of Lemma \ref{LEM.Wstar_CLRA_III}, we have $1)$ and our first claim in $2)$. Let $x\in L^{1}(M,\tau)_{+}$. For all $u\in L^{2}(M,\tau)$, $\|xu\|_{2}^{2}=\sup_{n\in\mathbb{N}}\|x_{n}u\|_{2}^{2}\in [0,\infty]$ by monotone convergence. Thus $u\in\dom L_{x}$ if and only if $\sup_{n\in\mathbb{N}}\|x_{n}u\|_{\tau}^{2}<\infty$, hence if and only if $\sup_{n\in\mathbb{N}}\|ux_{n}\|_{\tau}^{2}<\infty$ by $2)$ in Proposition \ref{PRP.Wstar_NCI_I}.
\end{proof}

\begin{cor}\label{COR.Wstar_CLRA_V}
For all $x\in L^{2}(M,\tau)_{h}$, $M\cap L^{2}(M,\tau)$ is core of $L_{x}$ and $R_{x}$.
\end{cor}
\begin{proof}
Since $x\in L^{0}(M,\tau)_{h}$, note $1)$ in Corollary \ref{COR.Wstar_CLRA_II} ensures it suffices to show our claim for $L_{x}$. Lemma \ref{LEM.Wstar_CLRA_III} lets us reduce further to showing $\dom L_{x}\cap L^{2}(M,\tau)_{+}$ lies in the closure of $M\cap L^{2}(A,\tau)$ w.r.t.~the graph norm of $L_{x}$.\par
Let $u\in\dom L_{x}\cap L^{2}(M,\tau)_{+}$ and set $u_{n}:=\min\{u,n\}\in L^{2}(M,\tau)$ for all $n\in\mathbb{N}$. For all $\lambda\geq 0$ and $n\in\mathbb{N}$, get $\min\{\lambda,n\}^{2}=\min\hspace{-0.033875cm} \lset{}\lambda^{2},n^{2}\rset\leq \lambda\cdot \min\{\lambda,n\}$. Multiplying out terms of the inner product lets us estimate

\begin{align}\label{EQ.COR.Wstar_CLRA_V_1}
\dblv{}u-u_{n}\dblv_{2}^{2}=\dblv{}u^{2}\dblv_{1}+\dblv{}\min\hspace{-0.046775cm} \big\{u^{2},n^{2}\big\}\dblv_{1}-2\tau\lc{}uu_{n}\rc\leq \dblv{}u^{2}\dblv_{1}-\dblv{}\min\hspace{-0.046775cm} \big\{u^{2},n^{2}\big\}\dblv_{1}
\end{align}

\noindent for all $n\in\mathbb{N}$. Equation \ref{EQ.COR.Wstar_CLRA_V_1} and $2)$ in Corollary \ref{COR.Wstar_CLRA_IV} imply $\lim_{n\in\mathbb{N}}\dblv{}u-u_{n}\dblv_{2}^{2}=0$.\par


\pagebreak


Note $x\in\dom R_{u}$ since $u\in\dom L_{x}$. For all $n\in\mathbb{N}$, get

\begin{align}\label{EQ.COR.Wstar_CLRA_V_2}
\dblv{}x\lc{}u-u_{n}\rc\dblv_{2}^{2}=\int_{\spec R_{u}}\big(\lambda-\min\{\lambda,n\}\big)^{2}dE_{R_{u}}^{x}<\infty.
\end{align}

\noindent Since $\lc\lambda-\min\{\lambda,n\}\rc^{2}\leq\lambda^{2}$ on $[0,\infty)$ for all $n\in\mathbb{N}$ by definition, applying Fatou's lemma to Equation \ref{EQ.COR.Wstar_CLRA_V_2} shows $\lim_{n\in\mathbb{N}}\dblv{}x\lc{}u-u_{n}\rc\dblv_{2}=0$.
\end{proof}


\subsubsection*{Spectral measures of self-adjoint measurable operators}

Using inverses of canonical left-~and right-actions, we extend Subsection \ref{SSEC.A_Fnd_FC} to self-adjoint measurable operators. This yields abstract notion of spectral and joint spectral measure, as well as bounded measurable functional and joint functional calculus of self-adjoint measurable operators. This subsumes Definition \ref{DFN.Wstar_FC}. Notation \ref{NTN.Wstar_CLRA_FC} fixes conventions.\par
Let $(M,\tau)$ be a tracial $W^{*}$-algebra.

\begin{dfn}\label{DFN.Wstar_CLRA_FC_I}
Let $x\in L^{0}(M,\tau)_{h}$.

\begin{itemize}
\item[1)] For all $Z\in\mathfrak{B}(\mathbb{R})$, set 

\begin{align}\label{EQ.DFN.Wstar_CLRA_FC_I_I}
E_{x,M}(Z):=L_{M}^{-1}\lc{}E_{L_{x,M}}(Z)\rc{}.    
\end{align}

\begin{reapply}
\end{reapply}

\noindent We call $E_{x,M}$ the spectral measure of $x$ in $M$.

\item[2)] The spectrum of $x$ in $M$ is $\specM x:=\spec L_{x,M}$. We call

\begin{align}\label{EQ.DFN.Wstar_CLRA_FC_I_II}
W_{M}^{*}(x):=L_{M}^{-1}\lc{}W^{*}\lc{}L_{x,M}\rc\rc{}   
\end{align}

\begin{reapply}
\end{reapply}

\noindent the $W^{*}$-algebra generated by $x$ in $M$.
\end{itemize}
\end{dfn}

\begin{prp}\label{PRP.Wstar_CLRA_FC}
If $x\in L^{0}(M,\tau)_{h}$, then

\begin{itemize}
\item[1)] $L_{E_{x,M}(Z),M}=E_{L_{x,M}}(Z)$ and $R_{E_{x,M}(Z),M}=E_{R_{x,M}}(Z)$ for all $Z\in\mathfrak{B}(\mathbb{R})$, 

\item[2)] $\specM x=\spec L_{x,M}=\spec R_{x,M}$,

\item[3)] $W_{M}^{*}(x)=L_{M}^{-1}\lc{}W_{M}^{*}\lc{}L_{x,M}\rc\rc{}=R_{M}^{-1}\lc{}W_{M}^{*}\lc{}R_{x,M}\rc\rc$.
\end{itemize}
\end{prp}
\begin{proof}
Let $x\in L^{0}(M,\tau)_{h}$. For all $Z\in\mathfrak{B}(\mathbb{R})$, we know $E_{x,M}(Z)\in M$. Thus our claim in $1)$ concerning $L_{M}$ holds. For $R_{M}$, we instead use $R_{x,M}=\Adj^{\dagger}L_{x,M}$ by $1)$ in Corollary \ref{COR.Wstar_CLRA_II} and reduce to $L_{M}$. We directly verify it suffices to show $\Adj^{\dagger}\lc{}E_{L_{x,M}}(Z)\rc{}=E_{\Adj^{\dagger}L_{x,M}}(Z)$ for all $Z\in\mathfrak{B}(\mathbb{R})$ to obtain our claim in $1)$ concerning $R_{M}$.\par
Since $\Adj^{\dagger}\lc{}R_{\pm i}(T)\rc{}=R_{\pm i}\lc\Adj^{\dagger}(T)\rc$, Lemma \ref{LEM.FC_Preservation_I} shows Lemma \ref{LEM.FC_Preservation_II} applies to $T=L_{x,M}$, $S=\Adj^{\dagger}(T)$ and $\phi=\Adj^{\dagger}$. Thus the required identity, hence $1)$ holds. Get $2)$ since the spectrum of a self-adjoint unbounded operator is the support of its spectral measure. Get $3)$ because all $W^{*}$-algebras involved are commutative.
\end{proof}

If $x\in L^{0}(M,\tau)_{h}$, then $\specM x$ is a locally compact Hausdorff space and with $\sigma$-ideal $\mathcal{N}\lc{}E_{x,M}\rc\subset\mathfrak{B}\lc\specM x\rc$ of null sets as per $1)$ in Definition \ref{DFN.Wstar_CLRA_FC_II}.

\begin{dfn}\label{DFN.Wstar_CLRA_FC_II}
Let $x\in L^{0}(M,\tau)_{h}$. Set

\begin{itemize}
\item[1)] $\mathcal{N}\lc{}E_{x,M}\rc{}:=\big\{\hspace{0.025cm} Z\in\mathfrak{B}(\mathbb{R})\ \vset\ E_{x,M}(Z)=0\hspace{0.025cm} \big\}$,

\item[2)] $L^{\infty}\lc\specM x,dE_{x,M}\rc{}:=L^{\infty}\lc\specM x,\mathcal{N}\lc{}E_{x,M}\rc\rc$.
\end{itemize}
\end{dfn}

\begin{lem}\label{LEM.Wstar_CLRA_FC}
If $x\in L^{0}(M,\tau)_{h}$, then

\begin{itemize}
\item[1)] $\lc{}L^{\infty}\lc\specM x,dE_{x,M}\rc{},\|.\|_{\infty}\rc$ is a $W^{*}$-algebra s.t.~$C_{0}\lc\specM x\rc$ is $\sigma$-weakly dense,

\item[2)] there exists normal unital $^{*}$-isomorphism

\begin{align}\label{EQ.LEM.Wstar_CLRA_FC_1}
\Gamma_{x,M}:L^{\infty}\lc\specM x,dE_{x,M}\rc\longrightarrow W_{M}^{*}(x)
\end{align}

\begin{reapply}
\end{reapply}

\noindent s.t.~$\Gamma_{L_{x,M}}=L_{M}\circ\Gamma_{x,M}$ and $\Gamma_{R_{x,M}}=R_{M}\circ\Gamma_{x,M}$,

\item[3)] $\Gamma_{x,M}$ is determined by unitality and

\begin{align}\label{EQ.LEM.Wstar_CLRA_FC_2}
\Gamma_{x,M}\big(\chi_{Z}\big)=E_{x,M}(Z)   
\end{align}

\begin{reapply}
\end{reapply}

\noindent for all $Z\in\mathfrak{B}(\mathbb{R})$.
\end{itemize}
\end{lem}
\begin{proof}
Note $L^{\infty}\lc\specM x,dE_{x,M}\rc{}=L^{\infty}\lc\spec L_{x,M},dE_{L_{x,M}}\rc$. Get $1)$. Proposition \ref{PRP.Wstar_CLRA_FC} implies $\Gamma_{L_{x,M}}=L_{M}\circ\Gamma_{x,M}$ and $\Gamma_{R_{x,M}}=R_{M}\circ\Gamma_{x,M}$. Proposition \ref{PRP.Wstar_Generated} and Proposition \ref{PRP.FC_Bd} thus show reducing to $\Gamma_{L_{x,M}}=L_{M}\circ\Gamma_{x,M}$ yields $2)$ and $3)$ in full.
\end{proof}

\begin{dfn}\label{DFN.Wstar_CLRA_FC_III}
Let $x\in L^{0}(M,\tau)_{h}$. We call $\Gamma_{x,M}$ the bounded measurable functional calculus of $x$ in $M$. For all $g\in L^{\infty}\lc\specM x,dE_{x,M}\rc$, set 

\begin{align}\label{EQ.DFN.Wstar_CLRA_FC_III_1}
g(x):=\Gamma_{x,M}(g).   
\end{align}
\end{dfn}

\begin{rem}
Let $x\in L^{0}(M,\tau)_{h}$. For all $g\in L^{\infty}\lc\specM x,dE_{x}\rc$, get $g(x)\in M\subset L^{0}(M,\tau)$ consistent with Definition \ref{DFN.Wstar_SMO_IV}. If $x\in M_{h}$, then we recover Definition \ref{DFN.Wstar_FC} since $3)$ in Lemma \ref{LEM.Wstar_CLRA_FC} reduces to $3)$ in Lemma \ref{LEM.Wstar_FC}.
\end{rem}

Let $x,y\in L^{0}(M,\tau)_{h}$. Using Proposition \ref{PRP.JFC_Strong_Com}, we know $2)$ in Lemma \ref{LEM.Wstar_CLRA_FC} at once implies $L_{x,M},R_{y,M}\in\UBII\lc{}L^{2}(M,\tau)\rc_{h}$ commute strongly. Equation \ref{EQ.SSEC.A_Fnd_FC_8} shows

\begin{align}\label{EQ.SSEC.B_SMO_CLRA_3}
W^{*}\lc{}L_{x,M}\rc\otimes W^{*}\lc{}R_{y,M}\rc{}=W^{*}\lc{}L_{x,M},R_{y,M}\rc\subset\BII\lc{}L^{2}(M,\tau)\rc{}.
\end{align}

\noindent Note Equation \ref{EQ.SSEC.B_SMO_CLRA_3} ensures Corollary \ref{COR.Wstar_TP} lets us tensor $L_{M}:W_{M}^{*}(x)\longrightarrow W^{*}\lc{}L_{x,M}\rc$ and $R_{M}:W_{M}^{*}(y)\longrightarrow W^{*}\lc{}R_{y,M}\rc$ to a normal unital $^{*}$-isomorphism 

\begin{align}\label{EQ.SSEC.B_SMO_CLRA_4}
L_{M}\otimes R_{M}:W_{M}^{*}(x)\otimes W_{M}^{*}(y)\longrightarrow W^{*}\lc{}L_{x,M},R_{y,M}\rc{}.
\end{align}

\begin{dfn}\label{DFN.Wstar_CLRA_FC_IV}
Let $x,y\in L^{0}(M,\tau)_{h}$.

\begin{itemize}
\item[1)] For all $Z\in\mathfrak{B}\lc\mathbb{R}\times\mathbb{R}\rc$, set 

\begin{align}\label{EQ.DFN.Wstar_CLRA_FC_IV_I}
E_{x,y,M}(Z):=\big(L_{M}\otimes R_{M}\big)^{-1}\lc{}E_{L_{x,M},R_{y,M}}(Z)\rc{}.
\end{align}

\begin{reapply}
\end{reapply}

\noindent We call $E_{x,y,M}$ the joint spectral measure of $x\otimes y$ in $M\otimes M^{\op}$.

\item[2)] The joint spectrum of $x\otimes y$ in $M\otimes M^{\op}$ is $\specM x\times y:=\spec L_{x,M}\times R_{y,M}$. We call

\begin{align}\label{EQ.DFN.Wstar_CLRA_FC_IV_II}
W_{M}^{*}(x,y):=W_{M}^{*}(x)\otimes W_{M}^{*}(y)=\big(L_{M}\otimes R_{M}\big)^{-1}\lc{}W^{*}\lc{}L_{x,M},R_{y,M}\rc\rc{}   
\end{align}

\begin{reapply}
\end{reapply}

\noindent the $W^{*}$-algebra generated by $x\otimes y$ in $M\otimes M^{\op}$.
\end{itemize}
\end{dfn}

If $x,y\in L^{0}(M,\tau)_{h}$, then $\specM x\times y$ is a locally compact Hausdorff space and with $\sigma$-ideal $\mathcal{N}\lc{}E_{x,y,M}\rc\subset\mathfrak{B}\lc\specM x\times y\rc$ of null sets as per $1)$ in Definition \ref{DFN.Wstar_CLRA_FC_V}.

\begin{dfn}\label{DFN.Wstar_CLRA_FC_V}
Let $x,y\in L^{0}(M,\tau)_{h}$. Set

\begin{itemize}
\item[1)] $\mathcal{N}\lc{}E_{x,y,M}\rc{}:=\big\{\hspace{0.025cm} Z\in\mathfrak{B}(\mathbb{R})\ \vset\ E_{x,y,M}(Z)=0\hspace{0.025cm} \big\}$,

\item[2)] $L^{\infty}\lc\specM x\times y,dE_{x,y,M}\rc{}:=L^{\infty}\lc\specM x\times y,\mathcal{N}\lc{}E_{x,y,M}\rc\rc$.
\end{itemize}
\end{dfn}

\begin{lem}\label{LEM.Wstar_CLRA_JFC}
If $x,y\in L^{0}(M,\tau)_{h}$, then

\begin{itemize}
\item[1)] $\lc{}L^{\infty}\lc\specM x\times y,dE_{x,y,M}\rc{},\|.\|_{\infty}\rc$ is a $W^{*}$-algebra s.t.~$C_{0}\lc\specM x\times y\rc$ is $\sigma$-weakly dense,

\item[2)] there exists normal unital $^{*}$-isomorphism

\begin{align}\label{EQ.LEM.Wstar_CLRA_JFC_1}
\Gamma_{x,y,M}:L^{\infty}\lc\specM x\times y,dE_{x,y,M}\rc\longrightarrow W_{M}^{*}(x,y)
\end{align}

\begin{reapply}
\end{reapply}

\noindent s.t.~$\Gamma_{L_{x,M},R_{y,M}}=\big(L_{M}\otimes R_{M}\big)\circ\Gamma_{x,M}$,

\item[3)] $\Gamma_{x,M}$ is determined by unitality and

\begin{align}\label{EQ.LEM.Wstar_CLRA_JFC_2}
\Gamma_{x,y,M}\lc\chi_{Z_{0}}\otimes\chi_{Z_{1}}\rc{}=E_{x,M}\lc{}Z_{0}\rc{}E_{y,M}\lc{}Z_{1}\rc{}
\end{align}

\begin{reapply}
\end{reapply}

\noindent for all $Z_{0},Z_{1}\in\mathfrak{B}(\mathbb{R})$.
\end{itemize}
\end{lem}
\begin{proof}
We know $1)$ by definition. For $2)$, note it reduces to factors by construction of joint spectral measures and apply Lemma \ref{LEM.Wstar_CLRA_FC}. Likewise get $3)$ by Proposition \ref{PRP.JFC_Bd}.
\end{proof}

\begin{dfn}\label{DFN.Wstar_CLRA_FC_VI}
Let $x,y\in L^{0}(M,\tau)_{h}$. We call $\Gamma_{x,y,M}$ the bounded measurable joint functional calculus of $x\otimes y$ in $M\otimes M^{\op}$. For all $g\in L^{\infty}\lc\specM x\times y,dE_{x,y,M}\rc$, set 

\begin{align}\label{EQ.DFN.Wstar_CLRA_FC_VI_1}
g(x,y):=\Gamma_{x,y,M}(g).    
\end{align}
\end{dfn}

\begin{ntn}\label{NTN.Wstar_CLRA_FC}
Unless stated otherwise, we suppress $W^{*}$-algebras in subscripts of spectral measures, spectra, bounded measurable functional calculus and generated\linebreak $W^{*}$-algebras. In general, we use subscripts to keep track of $W^{*}$-(sub-)algebras apart from the algebra of bounded operators on a Hilbert space.
\end{ntn}

\begin{lem}\label{LEM.Wstar_CLRA_JFC_SR}
Let $x,y\in L^{0}(M,\tau)_{+}$. If $g\in C_{b}\lc{}[0,\infty)\times [0,\infty)\rc$, then

\begin{align}\label{EQ.LEM.Wstar_CLRA_JFC_SR_1}
\Gamma_{x,y,M}(g)=\s\textrm{-}\lim_{\varepsilon\downarrow 0}\hspace{0.025cm} \Gamma_{x+\varepsilon 1_{M},y+\varepsilon 1_{M},M}(g).
\end{align}
\end{lem}
\begin{proof}
Note $2)$ in Lemma \ref{LEM.Wstar_CLRA_JFC} implies Equation \ref{EQ.LEM.Wstar_CLRA_JFC_SR_1} is equivalent to

\begin{align}\label{EQ.LEM.Wstar_CLRA_JFC_SR_2}
\Gamma_{L_{x},R_{y}}(g)=\s\textrm{-}\lim_{\varepsilon\downarrow 0}\hspace{0.025cm} \Gamma_{L_{x}+\varepsilon I,R_{y}+\varepsilon I}(g).
\end{align}

\noindent Let $\lset\varepsilon_{n}\rset_{n\in\mathbb{N}}\subset (0,\infty)$ be a descending sequence converging to zero. Proposition 10.1.8 in \cite{BK.deOli.2009.OpAlg_Quantum_Dynamics} implies $L_{x}=\sr$-$\lim_{n\in\mathbb{N}}L_{x}+\varepsilon_{n}I$ and $R_{y}=\sr$-$\lim_{n\in\mathbb{N}}R_{y}+\varepsilon_{n}I$. All unbounded operators used here are positive, and each limit is clearly independent of the given descending sequence. Using standard arguments, we see Lemma \ref{LEM.FC_SR} implies Equation \ref{EQ.LEM.Wstar_CLRA_JFC_SR_2}.
\end{proof}


\section{Compressed pull-back of joint functional calculus}\label{SEC.B_JFC}

In Subsection \ref{SSEC.B_JFC_L2Red}, we discuss semi-finite $W^{*}$-subalgebras of tracial $W^{*}$-algebras and associated $L^{2}$-reducible measurable operators. Assuming such semi-finiteness upon inclusion, f.s.n.~traces restrict to f.s.n.~traces. Theorem \ref{THM.Wstar_L2Red} gives structure-preserving canonical inclusion for spaces of measurable operators. These let us extend abstract compression maps from $W^{*}$-algebras to spaces of measurable operators.\par
In Subsection \ref{SSEC.B_JFC_Compression}, we formulate compressed pulled-back joint functional calculus of self-adjoint measurable operators. Theorem \ref{THM.JFC_Compression} states sufficient conditions. For its proof, we express change of canonical left-~and right-actions as abstract compression maps. We use Theorem \ref{THM.JFC_Compression} to define compressed pulled-back bounded measurable joint functional calculus of self-adjoint measurable operators and extend to suitable unbounded functions following its Corollary \ref{COR.JFC_Compression_I}.


\subsection[$L^{2}$-reducible measurable operators]{$\mathbf{L}^{2}$-reducible measurable operators}\label{SSEC.B_JFC_L2Red}

Semi-finite $W^{*}$-subalgebras are tracial $W^{*}$-algebras. We construct structure-preserving canonical inclusions of the resulting spaces of measurable operators in Theorem \ref{THM.Wstar_L2Red} by mapping to $L^{2}$-reducible measurable operators.\par
Inclusion of pre-duals ensure semi-finite $W^{*}$-algebras have unique noncommutative conditional expectations by dualisation \cite{BK.Tak.1979.OpAlg_I}. Abstract compression maps are one of two example classes. While we do not use them in the appendix, we do use noncommutative conditional expectations to show monotonicity of quasi-entropies in Subsection \ref{SSEC.NCDS_NCD_QE}.


\subsubsection*{Semi-finite $\mathbf{W^{*}}$-subalgebras}

Let $(M,\tau)$ be a tracial $W^{*}$-algebra. If $N\subset M$ is a $W^{*}$-subalgebra, then $\tau\vert_{N_{+}}$ is faithful normal trace on $N$ since $N_{+}\subset M_{+}$.

\begin{dfn}\label{DFN.Wstar_Trace_SF_Subalg}
If $N\subset M$ is a $W^{*}$-subalgebra s.t.~$\tau\vert_{N_{+}}$ is semi-finite, then we call $N$ a semi-finite $W^{*}$-subalgebra. We write $N\subset (M,\tau)$ in this case.
\end{dfn}

\begin{ntn}
Let $N\subset (M,\tau)$. We write $\tau=\tau\vert_{N_{+}}$ on $N$.
\end{ntn}

\begin{prp}\label{PRP.Wstar_Trace_SF_Subalg}
Let $N\subset M$ be a $W^{*}$-subalgebra.

\begin{itemize}
\item[1)] If $N\subset (M,\tau)$, then $\tau$ is f.s.n.~trace on $N$ and $(N,\tau)$ is a tracial $W^{*}$-algebra.

\item[2)] $N\subset (M,\tau)$ if and only if $N[1_{M}]\subset (M,\tau)$.

\item[3)] $N\subset (M,\tau)$ if and only if $N^{\op}\subset \lc{}M^{\op},\tau\rc$.
\end{itemize}
\end{prp}
\begin{proof}
We have $1)$ by definition. If $N\subset (M,\tau)$, then we know $N[1_{M}]=N\oplus\langle 1_{N}^{\perp}\rangle_{\mathbb{C}}\subset (M,\tau)$ by Proposition \ref{PRP.Wstar_Unitalisation}. If $\lc{}N[1_{M}],\tau\rc\subset (M,\tau)$, then $N\subset \lc{}N[1_{M}],\tau\rc$ shows $N\subset (M,\tau)$ at once. Get $2)$. We obtain $3)$ since partial orders on $N$ and $N^{\op}$ are identical.
\end{proof}

Let $N\subset (M,\tau)$. By construction, $\HII(N,\tau)\subset\HII(M,\tau)$. We define isometric inclusion $\inc_{2}:L^{2}(N,\tau)\longrightarrow L^{2}(M,\tau)$ of Hilbert spaces as the unique bounded linear map s.t.~we have commutative diagram

\smallskip

\begin{equation}\label{EQ.SSEC.B_JFC_L2Red_1}
\begin{tikzcd}
L^{2}(N,\tau)\arrow[rr,"\inc_{2}"] & & L^{2}(M,\tau) \\
& & \\
\HII(N,\tau)\arrow[rr,hook]\arrow[uu,swap,"\id_{\tau}"] & & \HII(M,\tau)\arrow[uu,swap,"\id_{\tau}"]
\end{tikzcd}
\end{equation}

\medskip

\noindent of Hilbert space isometries.

\begin{dfn}\label{DFN.L2_Closure_Temp}
For all $N\subset (M,\tau)$, set

\begin{align}\label{EQ.DFN.L2_Closure_Temp_1}
\mathbf{L}^{2}(N,\tau):=\overline{\textrm{inc}_{2}\lc{}L^{2}(N,\tau)\rc{}}^{\|.\|_{2}}=\textrm{inc}_{2}\lc{}L^{2}(N,\tau)\rc{}
\end{align}

\noindent for $\inc_{2}$ as per Diagram \ref{EQ.SSEC.B_JFC_L2Red_1}.
\end{dfn}

\begin{rem}
Note $\inc_{2}=\id_{L^{2}(N,\tau)}$ upon identifying as per Remark \ref{REM.Wstar_L2Red} following Theorem \ref{THM.Wstar_L2Red}, i.e.~a natural identification by extending Diagram \ref{EQ.PRP.Wstar_Trace_Compression_1} to $L^{0}(M,\tau)$.
\end{rem}

The isometric isomorphism $\inc_{2}:L^{2}(N,\tau)\longrightarrow\mathbf{L}^{2}(N,\tau)$ of Hilbert spaces defines $\inc_{2}^{\dagger}$ as per Definition \ref{DFN.Unbd_Twist}. We introduce compression maps in Subsection \ref{SSEC.A_Maps_Compression}.

\begin{prp}\label{PRP.Wstar_Trace_Compression}
Let $N\subset (M,\tau)$. We have commutative diagram

\smallskip

\begin{equation}\label{EQ.PRP.Wstar_Trace_Compression_1}
\begin{tikzcd}
N\arrow[rr,hook]\arrow[dd,"\id_{N}"] & & N[1_{M}]\arrow[rr,"L_{M}"]\arrow[dd,"\comunit"] & & \BII\lc{}L^{2}(M,\tau)\rc\arrow[dd,"\inc_{2}^{-\dagger}\circ\hspace{0.0675cm} \com_{\mathbf{L}^{2}(N,\tau)}"] \\
& & & & \\
N\arrow[rr,"\id_{N}"] & & N\arrow[rr,"L_{N}"] & & \BII\lc{}L^{2}(N,\tau)\rc{}
\end{tikzcd}
\end{equation}

\medskip

\noindent s.t.~horizontal maps are normal unital injective $^{*}$-homomorphisms and vertical ones are positivity-preserving surjections of Banach spaces.
\end{prp}
\begin{proof}
Diagram \ref{EQ.SSEC.A_Maps_Compression_1} is the left diagram in Diagram \ref{EQ.PRP.Wstar_Trace_Compression_1}. Get unital $W^{*}$-subalgebra $L_{M}(N)''=L_{M}\lc{}N[1_{M}]\rc\subset\BII\lc{}L^{2}(M,\tau)\rc$. Proposition \ref{PRP.Wstar_Equivalence} and Proposition \ref{PRP.Wstar_Trace_Ext_II} show

\begin{align}\label{EQ.PRP.Wstar_Trace_Compression_2}
N[1_{M}]=\lset{}x\in M\ \vset\ L_{x,M}\ \textrm{is}\ L_{M}(N)''\textrm{-affiliated}\rset{}.
\end{align}

\noindent For all $x\in N[1_{M}]_{h}$, Proposition \ref{PRP.Wstar_SMO_I} shows the affiliation property in Equation \ref{EQ.PRP.Wstar_Trace_Compression_2} lets us apply Corollary \ref{COR.Compression_Preservation_II} to get

\begin{align}\label{EQ.PRP.Wstar_Trace_Compression_3}
\lb{}L_{x,M},\pi_{\mathbf{L}^{2}\lc{}N[1_{M}],\tau\rc{}}\rb{}=0.
\end{align}

\noindent Equation \ref{EQ.PRP.Wstar_Trace_Compression_3} holds for all $x\in N[1_{M}]$ by decomposing into real and imaginary parts.\par
Note $N[1_{M}]=N\oplus\langle 1_{N}^{\perp}\rangle_{\mathbb{C}}$ using direct sum of $C^{*}$-algebras as per Proposition \ref{PRP.Wstar_Unitalisation} since $1_{N}^{\perp}=1_{M}-1_{N}$ by definition. Using $N1_{N}^{\perp}=1_{N}^{\perp}N=0$, Equation \ref{EQ.PRP.Wstar_Trace_Compression_2} shows

\begin{align}\label{EQ.PRP.Wstar_Trace_Compression_4}
N=\lset{}x\in M\ \vset\ L_{x,M}\ \textrm{is}\ L_{M}(N)''\textrm{-affiliated},\ x=1_{N}x\rset{}.
\end{align}

\noindent Using $\mathbf{L}^{2}(N,\tau)\subset\mathbf{L}^{2}\lc{}N[1_{M}],\tau\rc$, we directly verify

\begin{align}\label{EQ.PRP.Wstar_Trace_Compression_5}
\pi_{\mathbf{L}^{2}(N,\tau)}=L_{1_{N},M}\pi_{\mathbf{L}^{2}\lc{}N[1_{M}],\tau\rc{}}=\pi_{\mathbf{L}^{2}\lc{}N[1_{M}],\tau\rc{}}L_{1_{N},M}
\end{align}

\noindent by testing on the inner product. For all $x\in N[1_{M}]_{h}$, we apply $\lb{}x,1_{N}\rb{}=0$, Equation \ref{EQ.PRP.Wstar_Trace_Compression_3} and Equation \ref{EQ.PRP.Wstar_Trace_Compression_5} to calculate

\begin{align}\label{EQ.PRP.Wstar_Trace_Compression_6}
\lb{}L_{x,M},\pi_{\mathbf{L}^{2}(N,\tau)}\rb{}=0. 
\end{align}

\noindent Equation \ref{EQ.PRP.Wstar_Trace_Compression_6} holds for all $x\in N[1_{M}]$ by decomposing into real and imaginary parts.\par
We show the right diagram in Diagram \ref{EQ.PRP.Wstar_Trace_Compression_1}. We directly verify $\comunit (1_{N}^{\perp})=0$ and $\com_{\mathbf{L}^{2}(N,\tau)}\lc{}L_{1_{N}^{\perp},M}\rc{}=0$. We are left to consider $x\in N$. Let $x\in N$ and $u\in L^{2}(N,\tau)$. Thus

\begin{align}\label{EQ.PRP.Wstar_Trace_Compression_7}
L_{x,M}\lc\textrm{inc}_{2}(u)\rc{}=\textrm{inc}_{2}\lc{}L_{x,N}(u)\rc{}
\end{align}

\noindent by $\|.\|_{2}$-density. Equation \ref{EQ.PRP.Wstar_Trace_Compression_6} and Equation \ref{EQ.PRP.Wstar_Trace_Compression_7} let us calculate

\begin{align*}
\textrm{com}_{\mathbf{L}^{2}(N,\tau)}\hspace{0.05cm} L_{x,M}\lc\textrm{inc}_{2}(u)\rc{} & = \pi_{\mathbf{L}^{2}(N,\tau)}\lc{}L_{x,M}\lc\textrm{inc}_{2}(u)\rc\rc \phantom{\bigg)} \\
& = \pi_{\mathbf{L}^{2}(N,\tau)}\lc\textrm{inc}_{2}\lc{}L_{x,N}(u)\rc\rc \phantom{\bigg)} \\
& = \textrm{inc}_{2}\lc{}L_{x,N}(u)\rc{}. \phantom{\bigg)}
\end{align*}

\noindent Apply $\inc_{2}^{-1}$ to get the right diagram in Diagram \ref{EQ.PRP.Wstar_Trace_Compression_1}. Altogether, get Diagram \ref{EQ.PRP.Wstar_Trace_Compression_1}. We are left to show positivity-preservation. This follows from Proposition \ref{PRP.Compression_Abstract_Bd}, as well as Proposition \ref{PRP.Compression_Concrete} based on the former.
\end{proof}

Arguing as in the proof of Proposition V.2.36 in \cite{BK.Tak.1979.OpAlg_I}, we construct noncommutative conditional expectations of semi-finite $W^{*}$-algebras. We may use Theorem \ref{THM.Wstar_L2Red} below for this since noncommutative conditional expectations are not used in its proof.\par
Let $N\subset (M,\tau)$. We identify as per Remark \ref{REM.Wstar_L2Red} following Theorem \ref{THM.Wstar_L2Red}. Thus $N_{*}=L^{1}(N,\tau)\subset L^{1}(M,\tau)=M_{*}$ and $L^{2}(N,\tau)\subset L^{2}(M,\tau)$ by $L^{0}(N,\tau)\subset L^{0}(M,\tau)$. Dualising this inclusion map $\iota:N_{*}\longrightarrow M_{*}$ yields unique noncommutative conditional expectation from $M$ to $N$. Definition \ref{DFN.Wstar_Trace_NCE} gives its defining properties.

\begin{dfn}\label{DFN.Wstar_Trace_NCE}
Let $N\subset (M,\tau)$. We say that a normal unital map $P:M\longrightarrow N$ is a noncommutative conditional expectation from $M$ to $N$ if

\begin{itemize}
\item[1)] $P(x)=x$ for all $x\in N$, \hfill (Projection)

\item[2)] $P(x)=0$ implies $x=0$ for all $x\in M_{+}$, \hfill (Faithfulness)

\item[3)] $P(x)(y)=x(y)$ for all $x\in M$ and $y\in N_{*}$. \hfill (Trace identity)
\end{itemize}
\end{dfn}

\begin{rem}\label{REM.Wstar_Trace_NCE}
Following Remark \ref{REM.Wstar_NCI_MSP}, we use the modified standard pairing as per Definition \ref{DFN.Wstar_NCI_MSP} to have noncommutative $L^{1}$-spaces as pre-duals. The trace identity is equivalent to the following. For all $x\in M$ and $y\in L^{1}(N,\tau)$, we have

\begin{align}\label{EQ.REM.Wstar_Trace_NCE_1}
\tau\lc{}P(x)^{*}y\rc{}=\tau\lc{}P(x^{*})y\rc{}=\tau(x^{*}y).
\end{align}

\noindent Equation \ref{EQ.REM.Wstar_Trace_NCE_1} shows $P$ is unique if it exists. If $\tau<\infty$, then $M\subset L^{2}(M,\tau)\subset L^{1}(M,\tau)$ by H\"older and we have analogous chain of subspaces for $N$. We therefore see $\tau<\infty$ ensures $P$ extends to the Hilbert space projection $\pi_{N}^{M}:L^{2}(M,\tau)\longrightarrow L^{2}(N,\tau)$.
\end{rem}

\begin{dfn}\label{DFN.Wstar_Trace_NCE_MSP}
Let $N\subset (M,\tau)$ and let $\iota:N_{*}\longrightarrow M_{*}$ denote the canonical inclusion given by the modified standard pairing. We call $\pi_{N}^{M}:=\iota^{*}:M\longrightarrow N$ the noncommutative conditional expectation from $M$ to $N$.
\end{dfn}

\begin{prp}\label{PRP.Wstar_Trace_NCE_I}
If $N\subset (M,\tau)$, then $\pi_{N}^{M}:M\longrightarrow N$ is noncommutative conditional expectation from $M$ to $N$. If $N\subset M$ is furthermore a unital $W^{*}$-subalgebra, then $\pi_{N}^{M}$ is trace-preserving.
\end{prp}
\begin{proof}
We may argue here as in the proof of Proposition V.2.36 in \cite{BK.Tak.1979.OpAlg_I} to show $\pi_{N}^{M}$ is a noncommutative conditional expectation. This assumes unitality. However, the latter is only used to show $\tau\circ\pi_{N}^{M}=\tau$ on $\mathfrak{m}_{\tau}$. Equation \ref{EQ.REM.Wstar_Trace_NCE_1} implies uniqueness.
\end{proof}

We give two classes of noncommutative conditional expectations used throughout our discussion in Proposition \ref{PRP.Wstar_Trace_NCE_II}. First, we decompose Hilbert space projections. We use such decomposition in order to reduce non-unital to unital cases if the given trace is finite. Secondly, we compress with projections using abstract compression maps as per Definition \ref{DFN.Compression_Abstract_Bd} for compressed $W^{*}$-subalgebras as per Example \ref{BSP.Wstar_CP_III}.

\begin{rem}
Assume $\tau<\infty$. If $N\subset M$ is a $W^{*}$-subalgebra, then we know $N\subset (M,\tau)$ by Proposition \ref{PRP.Wstar_Trace_Fin_II}. The latter shows semi-finiteness is satisfied for finite faithful normal traces. We use this for Definition \ref{DFN.Wstar_Trace_NCE_Kappa} and $1)$ in Proposition \ref{PRP.Wstar_Trace_NCE_II}.
\end{rem}

\begin{dfn}\label{DFN.Wstar_Trace_NCE_Kappa}
Assume $\tau<\infty$. Let $N\subset M$ be a $W^{*}$-subalgebra. For all $x\in M$, set

\begin{align*}
\kappa_{N}^{M}(x):=
\begin{cases}\tau(1_{N}^{\perp})^{-1}\cdot \tau\vstretch{0.9375}{\bigg(}\pi_{\langle 1_{N}^{\perp}\rangle_{\mathbb{C}}}^{M}(x)\vstretch{0.9375}{\bigg)} & \If\ 1_{M}\neq 1_{N}, \\
0 & \Else.
\end{cases}
\end{align*}
\end{dfn}

\begin{prp}\label{PRP.Wstar_Trace_NCE_II}
Let $N\subset M$ be a $W^{*}$-subalgebra.

\begin{itemize}
\item[1)] If $\tau<\infty$, then $N\subset (M,\tau)$ and $\pi_{N}^{M}=\pi_{N[1_{M}]}^{M}-\kappa_{N}^{M}1_{N}^{\perp}$.

\item[2)] If $p\in M$ is a projection, then $M[p]\subset (M,\tau)$ and $\pi_{M[p]}^{M}=\comp$.
\end{itemize} 
\end{prp}
\begin{proof}
Assume $\tau<\infty$. Proposition \ref{PRP.Wstar_Trace_Fin_II} shows $N\subset (M,\tau)$. By our construction of noncommutative $L^{2}$-spaces, we have orthogonal decomposition

\begin{align}\label{EQ.PRP.Wstar_Trace_NCE_II_1}
L^{2}\lc{}N[1_{M}],\tau\rc{}=L^{2}(N,\tau)\oplus\langle 1_{N}^{\perp}\rangle_{\mathbb{C}}\subset L^{2}(M,\tau)    
\end{align}

\noindent since $N[1_{M}]=N\oplus\langle 1_{N}^{\perp}\rangle_{\mathbb{C}}$ by Proposition \ref{PRP.Wstar_Unitalisation}. Extending to Hilbert space projections as per Remark \ref{REM.Wstar_Trace_NCE}, Equation \ref{EQ.PRP.Wstar_Trace_NCE_II_1} shows

\begin{align}\label{EQ.PRP.Wstar_Trace_NCE_II_2}
\pi_{N[1_{M}]}^{M}=\pi_{N}^{M}\oplus\pi_{\langle 1_{N}^{\perp}\rangle_{\mathbb{C}}}^{M}
\end{align}

\noindent w.r.t.~$\BII\lc{}L^{2}(N,\tau)\rc\oplus\BII\lc\langle 1_{N}^{\perp}\rangle_{\mathbb{C}}\rc$. Equation \ref{EQ.PRP.Wstar_Trace_NCE_II_2} implies $1)$ at once.\par


\pagebreak


We show $2)$. Let $p\in M$ be a projection. Let $x\in N_{+}$ be non-zero. Semi-finiteness of $\tau$ yields $y\in M_{+}$ s.t.~$y\leq x$ and $\tau(y)<\infty$. Note $pyp\in M[p]_{+}$. Get $pyp\leq x$ in $M[p]$ since $x=pxp$. In addition, traciality implies

\begin{align}\label{EQ.PRP.Wstar_Trace_NCE_II_3}
0\leq\tau(pyp)+\tau\big(\lc{}1_{M}-p\rc{}y\lc{}1_{M}-p\rc\big)=\tau(y)<\infty.    
\end{align}

\noindent Equation \ref{EQ.PRP.Wstar_Trace_NCE_II_3} implies $\tau(pyp)<\infty$. We obtain $M[p]\subset (M,\tau)$. Equation \ref{EQ.REM.Wstar_Trace_NCE_1} determines noncommutative conditional expectations. Moreover, $M[p]\cap L^{1}\lc{}M[p],\tau\rc\subset L^{1}\lc{}M[p],\tau\rc$ is $\|.\|_{1}$-dense by construction as per Definition \ref{DFN.NC_Int_Lp}. It suffices to show 

\begin{align}\label{EQ.PRP.Wstar_Trace_NCE_II_4}
\tau\lc\pi_{M[p]}^{M}(x)^{*}y\rc{}=\tau\lc\lc\comp x\rc^{*}y\rc{}    
\end{align}

\noindent for all $x\in M$ and $y\in M[p]\cap L^{1}\lc{}M[p],\tau\rc$. Applying Equation \ref{EQ.REM.Wstar_Trace_NCE_1} and using $y=pyp$ in each case, we directly verify Equation \ref{EQ.PRP.Wstar_Trace_NCE_II_4}. Get $2)$.
\end{proof}

In Subsection \ref{SSEC.NCDS_AF_BIM}, we write noncommutative conditional expectations in the unital finite-dimensional case as as averages of unitary conjugations. We ultimately obtain the general non-unital finite-dimensional one by $1)$ in Proposition \ref{PRP.Wstar_Trace_NCE_II}.


\subsubsection*{$\mathbf{L}^{2}$-reducible measurable operators}

Let $(M,\tau)$ be a tracial $W^{*}$-algebra. For all $N\subset (M,\tau)$, Theorem \ref{THM.Wstar_L2Red} yields canonical inclusion $L^{0}(N,\tau)\subset L^{0}(M,\tau)$ preserving noncommutative $L^{p}$-norms. Equation \ref{EQ.PRP.Wstar_Trace_Compression_4} leads to Definition \ref{DFN.Wstar_L2Red_I}.

\begin{dfn}\label{DFN.Wstar_L2Red_I}
For all $N\subset (M,\tau)$, we call

\begin{align}\label{EQ.DFN.Wstar_L2Red_I_1}
\mathbf{L}^{0}(N,\tau):=\lset{}x\in L^{0}(M,\tau)\ \vset\ L_{x,M}\ \textrm{is}\ L_{M}(N)''\textrm{-affiliated},\ x=1_{N}x\rset{}
\end{align}

\noindent the space of $L^{2}(N,\tau)$-reducible measurable operators in $L^{0}(M,\tau)$.
\end{dfn}

\begin{prp}\label{PRP.Wstar_L2Red_I}
Let $N\subset (M,\tau)$.

\begin{itemize}
\item[1)] $\mathbf{L}^{0}(N,\tau)\subset\mathbf{L}^{0}\lc{}N[1_{M}],\tau\rc\subset L^{0}(M,\tau)$ are $^{*}$-subalgebras.

\item[2)] $\mathbf{L}^{0}(N,\tau)=\overline{N}$ for uniform closure in measure topology of $(M,\tau)$.
\end{itemize}
\end{prp}
\begin{proof}
Note $L_{M}(N)''=L_{M}\lc{}N[1_{M}]\rc$. The construction of spaces of measurable operators reviewed in Subsection \ref{SSEC.B_SMO_NCI} taken from \cite{BK.Tak.2003.OpAlg_II} is in fact independent of choice of normal faithful unital $^{*}$-representation. Using $L_{M}$ and f.s.n.~trace $\tau:L_{M}\lc{}N[1_{M}]\rc\longrightarrow [0,\infty]$, we see $L_{M}$ maps $\mathbf{L}^{0}\lc{}N[1_{M}],\tau\rc$ onto $L^{0}\lc{}L_{M}\lc{}N[1_{M}]\rc{},\tau\rc$. This implies $1)$ and $2)$ for $N[1_{M}]$ since uniform structure is determined by the measure topology on $M$, resp.~$\LII(M)$.\par
By definition, $\mathbf{L}^{0}(N,\tau)\subset L^{0}\lc{}N[1_{M}],\tau\rc$ is a $^{*}$-subalgebra. Get $1)$. Equation \ref{EQ.PRP.Wstar_Trace_Compression_4} shows $N\subset\mathbf{L}^{0}(N,\tau)$. We have $\mathbf{L}^{0}(N,\tau)\subset\overline{N}$ by $2)$ for $N[1_{M}]$ and continuity of multiplication on bounded subsets of $L^{0}(M,\tau)$ \lc{}cf.~Theorem IX.2.2 in \cite{BK.Tak.2003.OpAlg_II} and \cite{ART.Nel.1974.Wstar_Integration}\rc{}. We therefore get $2)$ by taking uniform closure.
\end{proof}

Let $N\subset (M,\tau)$. Let $x\in\mathbf{L}^{0}(N,\tau)$ be self-adjoint. For all $Z\in\mathfrak{B}(\mathbb{R})$,  Proposition \ref{PRP.Wstar_SMO_I} ensures the affiliation property in Equation \ref{EQ.DFN.Wstar_L2Red_I_1} implies $E_{L_{x,M}}(Z)\in L_{M}\lc{}N[1_{M}]\rc$. For all $Z\in\mathfrak{B}(\mathbb{R})$, get $E_{x,M}(Z)\in N[1_{M}]$ by $1)$ in Proposition \ref{PRP.Wstar_CLRA_FC} and we have decomposition

\begin{align}\label{EQ.SSEC.B_JFC_L2Red_2}
E_{x,M}(Z)=1_{N}E_{x,M}(Z)1_{N}\oplus 1_{N}^{\perp}E_{x,M}(Z)1_{N}^{\perp}=\comunit E_{x,M}(Z)\oplus\nu_{x,N}(Z)1_{N}^{\perp}
\end{align}

\noindent w.r.t.~$N[1_{M}]=N\oplus\langle 1_{N}^{\perp}\rangle_{\mathbb{C}}$. Note $\nu_{x,N}(Z)\in\lset{}0,1\rset$ in each case. Equation \ref{EQ.SSEC.B_JFC_L2Red_2} in turn yields two compressed spectral measures. For all $Z\in\mathfrak{B}(\mathbb{R})$, set

\begin{align}\label{EQ.SSEC.B_JFC_L2Red_3}
E_{x,N}(Z):=\comunit E_{x,M}(Z).
\end{align}

\noindent The map $Z\mapsto\nu_{x,N}(Z)\in\lset{}0,1\rset$ defined on $\mathfrak{B}(\mathbb{R})$ is determined by $Z\mapsto 1_{N}^{\perp}E_{x,M}(Z)1_{N}^{\perp}$. If $N\subset M$ is a unital $W^{*}$-subalgebra, then set $\nu_{x}:=0$. If not, then $E_{x,M}$ spectral measure of $x$ in $M$ and $\nu_{x,N}(Z)\in\lset{}0,1\rset$ in each case implies there exists unique $\nu_{x}\in\mathbb{R}$ s.t.~

\begin{align}\label{EQ.SSEC.B_JFC_L2Red_4}
\nu_{x,N}(Z)=\chi_{Z}(\nu_{x})
\end{align}

\noindent for all $Z\in\mathfrak{B}(\mathbb{R})$. Equation \ref{EQ.SSEC.B_JFC_L2Red_4} shows $\nu_{x}$ determines $\nu_{x,N}$.

\begin{dfn}\label{DFN.Wstar_L2Red_II}
Let $N\subset (M,\tau)$. For all self-adjoint $x\in\mathbf{L}^{0}(N,\tau)$, we define

\begin{itemize}
\item[1)] the map $Z\mapsto E_{x,N}(Z)$ on $\mathfrak{B}(\mathbb{R})$ as per Equation \ref{EQ.SSEC.B_JFC_L2Red_3},

\item[2)] $\nu_{x}=0$ if $N\subset M$ is a unital $W^{*}$-subalgebra, and $\nu_{x}\in\mathbb{R}$ as per Equation \ref{EQ.SSEC.B_JFC_L2Red_4} if not.
\end{itemize}
\end{dfn}

\begin{rem}
Upon identifying $\mathbf{L}^{0}(N,\tau)=L^{0}(N,\tau)$ as per Remark \ref{REM.Wstar_L2Red}, we readily see $E_{x,N}$ as per $1)$ in Definition \ref{DFN.Wstar_L2Red_II} is in fact the spectral measure of $x$ in $N$ as per $1)$ in Definition \ref{DFN.Wstar_CLRA_FC_I} for all $x\in L^{0}(N,\tau)$.
\end{rem}

\begin{prp}\label{PRP.Wstar_L2Red_II}
Let $N\subset (M,\tau)$. If $x\in\mathbf{L}^{0}\lc{}N[1_{M}],\tau\rc$ is self-adjoint, then

\begin{itemize}
\item[1)] we define spectral measure $L_{N}\lc{}E_{x,N}\rc$ on $\mathbb{R}$ with values in $\BII\lc{}L^{2}(N,\tau)\rc$ by setting

\begin{align}\label{EQ.PRP.Wstar_L2Red_II_1}
L_{N}\lc{}E_{x,N}\rc(Z):=L_{N}\lc{}E_{x,N}(Z)\rc{}    
\end{align}

\begin{reapply}
\end{reapply}

\noindent for all $Z\in\mathfrak{B}(\mathbb{R})$, 

\item[2)] $L_{x,M}$ is $\mathbf{L}^{2}\lc{}N[1_{M}],\tau\rc$-, $\mathbf{L}^{2}(N,\tau)$-~and $\mathbf{L}^{2}(\langle 1_{N}^{\perp}\rangle_{\mathbb{C}},\tau)$-reducible,

\item[3)] $\com_{\mathbf{L}^{2}\lc{}N[1_{M}],\tau\rc{}}L_{x,M}=\com_{\mathbf{L}^{2}(N,\tau)}L_{x,M}+\com_{\mathbf{L}^{2}(\langle 1_{N}^{\perp}\rangle_{\mathbb{C}},\tau)}L_{x,M}$.
\end{itemize}
\end{prp}
\begin{proof}
Let $x\in\mathbf{L}^{0}\lc{}N[1_{M}],\tau\rc$ be self-adjoint. Get $1)$ by Proposition \ref{PRP.Wstar_Trace_Compression}. We show $2)$. For all $Z\in\mathfrak{B}(\mathbb{R})$, we use $1)$ in Proposition \ref{PRP.Wstar_CLRA_FC} and Equation \ref{EQ.PRP.Wstar_Trace_Compression_3} to calculate

\begin{align}\label{EQ.PRP.Wstar_L2Red_II_2}
\lb{}E_{L_{x,M}}(Z),\pi_{\mathbf{L}^{2}\lc{}N[1_{M}],\tau\rc{}}\rb{}=\lb{}L_{E_{x,M}(Z),M},\pi_{\mathbf{L}^{2}\lc{}N[1_{M}],\tau\rc{}}\rb{}=0.
\end{align}

\noindent Equation \ref{EQ.PRP.Wstar_L2Red_II_2} shows $L_{x,M}$ is $\mathbf{L}^{2}\lc{}N[1_{M}],\tau\rc$-reducible by Corollary \ref{COR.Compression_Preservation_II}. If we instead use Equation \ref{EQ.PRP.Wstar_Trace_Compression_4}, then we calculate

\begin{align}\label{EQ.PRP.Wstar_L2Red_II_3}
\lb{}E_{L_{x,M}}(Z),\pi_{\mathbf{L}^{2}(N,\tau)}\rb{}=\lb{}L_{E_{x,M}(Z),M},\pi_{\mathbf{L}^{2}(N,\tau)}\rb{}=0
\end{align}

\noindent in each case. Equation \ref{EQ.PRP.Wstar_L2Red_II_3} implies $L_{x,M}$ is $\mathbf{L}^{2}(N,\tau)$-reducible by Corollary \ref{COR.Compression_Preservation_II}. If $\tau(1_{N}^{\perp})=\infty$, then $\mathbf{L}^{2}(\langle 1_{N}^{\perp}\rangle_{\mathbb{C}},\tau)=0$ by construction. If not, then $\mathbf{L}^{2}(\langle 1_{N}^{\perp}\rangle_{\mathbb{C}},\tau)=\langle 1_{N}^{\perp}\rangle_{\mathbb{C}}$. For all $Z\in\mathfrak{B}(\mathbb{R})$, we obtain

\begin{align}\label{EQ.PRP.Wstar_L2Red_II_4}
E_{L_{x,M}}(Z)\pi_{\mathbf{L}^{2}(\langle 1_{N}^{\perp}\rangle_{\mathbb{C}},\tau)}=\pi_{\mathbf{L}^{2}(\langle 1_{N}^{\perp}\rangle_{\mathbb{C}},\tau)}E_{L_{x,M}}(Z)=\nu_{x,N}(Z)\cdot \pi_{\mathbf{L}^{2}(\langle 1_{N}^{\perp}\rangle_{\mathbb{C}},\tau)}.
\end{align}

\noindent Equation \ref{EQ.PRP.Wstar_L2Red_II_4} shows $L_{x,M}$ is $\mathbf{L}^{2}(\langle 1_{N}^{\perp}\rangle_{\mathbb{C}},\tau)$-reducible by Corollary \ref{COR.Compression_Preservation_II}. Get $2)$.\par
We show $3)$. Using $2)$, $1.3)$ in Proposition \ref{PRP.Reducible} shows

\begin{align}\label{EQ.PRP.Wstar_L2Red_II_5}
L_{x,M}=\textrm{com}_{\mathbf{L}^{2}(N,\tau)}\hspace{0.05cm} L_{x,M}+\textrm{com}_{\mathbf{L}^{2}(N,\tau)^{\perp}}\hspace{0.05cm} L_{x,M}.
\end{align}

\noindent Applying $\com_{\mathbf{L}^{2}\lc{}N[1_{M}],\tau\rc{}}$ to Equation \ref{EQ.PRP.Wstar_L2Red_II_5} yields

\begin{align}\label{EQ.PRP.Wstar_L2Red_II_6}
\textrm{com}_{\mathbf{L}^{2}\lc{}N[1_{M}],\tau\rc{}}\hspace{0.05cm} L_{x,M}=\textrm{com}_{\mathbf{L}^{2}(N,\tau)}\hspace{0.05cm} L_{x,M}+\textrm{com}_{\mathbf{L}^{2}\lc{}N[1_{M}],\tau\rc{}}\hspace{0.05cm} \lc\textrm{com}_{\mathbf{L}^{2}(N,\tau)^{\perp}}\hspace{0.05cm} L_{x,M}\rc{}
\end{align}

\noindent since $\mathbf{L}^{2}(N,\tau)\subset\mathbf{L}^{2}\lc{}N[1_{M}],\tau\rc$. We directly verify

\begin{align}\label{EQ.PRP.Wstar_L2Red_II_7}
\pi_{\mathbf{L}^{2}(\langle 1_{N}^{\perp}\rangle_{\mathbb{C}},\tau)}=\pi_{\mathbf{L}^{2}(N,\tau)^{\perp}}\pi_{\mathbf{L}^{2}\lc{}N[1_{M}],\tau\rc{}}=\pi_{\mathbf{L}^{2}\lc{}N[1_{M}],\tau\rc{}}\pi_{\mathbf{L}^{2}(N,\tau)^{\perp}}
\end{align}

\noindent by testing on the inner product. Equation \ref{EQ.PRP.Wstar_L2Red_II_7} implies

\begin{align}\label{EQ.PRP.Wstar_L2Red_II_8}
\textrm{com}_{\mathbf{L}^{2}\lc{}N[1_{M}],\tau\rc{}}\hspace{0.05cm} \lc\textrm{com}_{\mathbf{L}^{2}(N,\tau)^{\perp}}\hspace{0.05cm} L_{x,M}\rc{}=\textrm{com}_{\mathbf{L}^{2}(\langle 1_{N}^{\perp}\rangle_{\mathbb{C}},\tau)}\hspace{0.05cm} L_{x,M}.
\end{align}

\noindent Applying Equation \ref{EQ.PRP.Wstar_L2Red_II_8} to the right-hand side of Equation \ref{EQ.PRP.Wstar_L2Red_II_6} shows $4)$.
\end{proof}


\pagebreak


\begin{lem}\label{LEM.Wstar_L2Red_I}
Let $N\subset (M,\tau)$. If $x\in\mathbf{L}^{0}\lc{}N[1_{M}],\tau\rc$ is self-adjoint, then $\int\lambda dL_{N}\lc{}E_{x,N}\rc$ is a $\tau$-measurable self-adjoint unbounded operator on $L^{2}(N,\tau)$.
\end{lem}
\begin{proof}
Let $x\in\mathbf{L}^{0}\lc{}N[1_{M}],\tau\rc$ be self-adjoint. Set $T_{x}:=\int\lambda dL_{N}\lc{}E_{x,N}\rc$. Proposition \ref{PRP.Wstar_SMO_I} shows the affiliation property in Equation \ref{EQ.DFN.Wstar_L2Red_I_1} ensures $T_{x}$ is $N$-affiliated. We are left to show $\tau$-measurability as claimed. We use Notation \ref{NTN.Wstar_Trace}.\par
Let $Z\in\mathfrak{B}(\mathbb{R})$. Proposition \ref{PRP.Wstar_Trace_Compression} shows

\begin{align}\label{EQ.LEM.Wstar_L2Red_I_1}
L_{N}\lc{}E_{x,N}(Z)\rc{}=\textrm{inc}_{2}^{-\dagger}\lc\textrm{com}_{\mathbf{L}^{2}(N,\tau)}\hspace{0.05cm} L_{M}\lc{}E_{x,M}(Z)\rc\rc{}.
\end{align}

\noindent Note $2)$ in Proposition \ref{PRP.Wstar_L2Red_II} implies $L_{x,M}$ is $\mathbf{L}^{2}(N,\tau)$-reducible. Lemma \ref{LEM.Compression_Preservation_I} shows

\begin{align}\label{EQ.LEM.Wstar_L2Red_I_2}
\textrm{com}_{\mathbf{L}^{2}(N,\tau)}\hspace{0.05cm} L_{M}\lc{}E_{x,M}(Z)\rc{}=\textrm{com}_{\mathbf{L}^{2}(N,\tau)}\hspace{0.05cm} L_{E_{x,M}(Z),M}=E_{\textrm{com}_{\mathbf{L}^{2}(N,\tau)}\hspace{0.05cm} L_{x,M}}(Z).
\end{align}

\noindent We combine Equation \ref{EQ.LEM.Wstar_L2Red_I_1} and Equation \ref{EQ.LEM.Wstar_L2Red_I_2} to

\begin{align}\label{EQ.LEM.Wstar_L2Red_I_3}
L_{N}\lc{}E_{x,N}(Z)\rc{}=\textrm{inc}_{2}^{-\dagger}\lc{}E_{\textrm{com}_{\mathbf{L}^{2}(N,\tau)}\hspace{0.05cm} L_{x,M}}(Z)\rc{}.
\end{align}

Upon inversion of $\inc_{2}^{\dagger}$, we see Equation \ref{EQ.LEM.Wstar_L2Red_I_3} ensures $2)$ in Corollary \ref{COR.Unbd_Twist_RP} applies here. Applying said corollary accordingly, get

\begin{align}\label{EQ.LEM.Wstar_L2Red_I_4}
T_{x}=\textrm{inc}_{2}^{-\dagger}\lc\textrm{com}_{\mathbf{L}^{2}(N,\tau)}\hspace{0.05cm} L_{x,M}\rc{}.
\end{align}

\noindent Equation \ref{EQ.LEM.Wstar_L2Red_I_4} implies $T_{x}^{2}=T_{x^{2}}$ and therefore

\begin{align}\label{EQ.LEM.Wstar_L2Red_I_5}
\absv{1.15}{T_{x}}=T_{\absv{0.925}{x}}.
\end{align}

\noindent Equation \ref{EQ.SSEC.B_JFC_L2Red_2} and Equation \ref{EQ.LEM.Wstar_L2Red_I_5} let us calculate

\begin{align}\label{EQ.LEM.Wstar_L2Red_I_6}
\tau\lc{}E_{\absv{0.925}{x},M}(Z)\rc{}=\tau\lc{}L_{M}\lc{}E_{\absv{0.925}{x},M}(Z)\rc\rc{}=\tau\lc{}E_{\absv{1.15}{T_{x}},M}(Z)\rc{}+\nu_{\absv{0.925}{x},N}(Z)\cdot \tau(1_{N}^{\perp})
\end{align}

\noindent for all $Z\in\mathfrak{B}(\mathbb{R})$. Equation \ref{EQ.LEM.Wstar_L2Red_I_6} shows $\tau$-measurability of $L_{x,M}$ implies $\tau$-measurability of $T_{x}$. This is our claim by construction.
\end{proof}

\begin{dfn}\label{DFN.Wstar_L2Red_Compression}
Let $N\subset (M,\tau)$. For all $x\in\mathbf{L}^{0}\lc{}N[1_{M}],\tau\rc$, set

\begin{align}\label{EQ.DFN.Wstar_L2Red_Compression_1}
\comN x:=L_{N}^{-1}\lc\int\lambda dL_{N}\lc{}E_{\RE(x),N}\rc\rc{}+iL_{N}^{-1}\lc\int\lambda dL_{N}\lc{}E_{\IM(x),N}\rc\rc{}.
\end{align}
\end{dfn}


\pagebreak


\begin{lem}\label{LEM.Wstar_L2Red_II}
Let $N\subset (M,\tau)$. If $x\in\mathbf{L}^{0}\lc{}N[1_{M}],\tau\rc$ is self-adjoint, then

\begin{itemize}
\item[1)] $\com_{\mathbf{L}^{2}\lc{}N[1_{M}],\tau\rc{}}L_{x,M}=\com_{\mathbf{L}^{2}(N,\tau)}\inc_{2}^{\dagger}\lc{}L_{\comN x,N}\rc{}+\nu_{x}\cdot \pi_{\mathbf{L}^{2}(\langle 1_{N}^{\perp}\rangle_{\mathbb{C}},\tau)}$.

\item[2)] $\com_{\mathbf{L}^{2}(N,\tau)}L_{x,M}=\com_{\mathbf{L}^{2}(N,\tau)}\inc_{2}^{\dagger}\lc{}L_{\comN x,N}\rc$,

\item[3)] $\com_{\mathbf{L}^{2}(\langle 1_{N}^{\perp}\rangle_{\mathbb{C}},\tau)}L_{x,M}=\nu_{x}\cdot \pi_{\mathbf{L}^{2}(\langle 1_{N}^{\perp}\rangle_{\mathbb{C}},\tau)}$.
\end{itemize}
\end{lem}
\begin{proof}
If we have $2)$ and $3)$, then $3)$ in Proposition \ref{PRP.Wstar_L2Red_II} implies $1)$. Get $2)$ by applying $\inc_{2}^{\dagger}$ to Equation \ref{EQ.LEM.Wstar_L2Red_I_4}. Using Lemma \ref{LEM.Compression_Preservation_I}, Equation \ref{EQ.PRP.Wstar_L2Red_II_4} shows the spectral theorem implies $3)$ since the given spectral measures coincide.
\end{proof}

\begin{cor}\label{COR.Wstar_L2Red_I}
Let $N\subset (M,\tau)$. For all $x\in\mathbf{L}^{0}\lc{}N[1_{M}],\tau\rc$, we have

\begin{itemize}
\item[1)] $1_{N}x1_{N}\in\mathbf{L}^{0}(N,\tau)$,

\item[2)] $\comN x=\comN 1_{N}x1_{N}$.
\end{itemize}
\end{cor}
\begin{proof}
We know $1)$ since the affiliation property in Equation \ref{EQ.DFN.Wstar_L2Red_I_1} is identical. Note $1)$ in Proposition \ref{PRP.Wstar_L2Red_I} shows all claims reduce to self-adjoint elements. Let $x\in\mathbf{L}^{0}\lc{}N[1_{M}],\tau\rc$ be self-adjoint. Using Corollary \ref{COR.Wstar_CLRA_III}, Equation \ref{EQ.PRP.Wstar_Trace_Compression_5} lets us calculate

\begin{align*}
\textrm{com}_{\mathbf{L}^{2}(N,\tau)}\hspace{0.05cm} L_{x,M} & = \textrm{com}_{\mathbf{L}^{2}\lc{}N[1_{M}],\tau\rc{}}\hspace{0.05cm} \overline{L_{1_{N},M}L_{x,M}L_{1_{N},M}} \phantom{\bigg)} \\
& = \textrm{com}_{\mathbf{L}^{2}\lc{}N[1_{M}],\tau\rc{}}\hspace{0.05cm} L_{1_{N},M}\cdot \overline{L_{x,M}L_{1_{N},M}} \phantom{\bigg)} \\
& = \textrm{com}_{\mathbf{L}^{2}\lc{}N[1_{M}],\tau\rc{}}\hspace{0.05cm} L_{1_{N},M}\cdot \overline{L_{x,M}}\cdot L_{1_{N},M} \phantom{\bigg)} \\
& = \textrm{com}_{\mathbf{L}^{2}\lc{}N[1_{M}],\tau\rc{}}\hspace{0.05cm} L_{1_{N},M}L_{x,M}L_{1_{N},M} \phantom{\bigg)} \\
& = \textrm{com}_{\mathbf{L}^{2}\lc{}N[1_{M}],\tau\rc{}}\hspace{0.05cm} L_{1_{N}x1_{N},M} \phantom{\bigg)}
\end{align*}

\noindent by boundedness of left-~and right-multiplication with $1_{N}$. Equation \ref{EQ.LEM.Wstar_L2Red_I_4} shows applying $\com_{\mathbf{L}^{2}(N,\tau)}$ to both sides of the above calculation yields $2)$.
\end{proof}

Upon identifying $\mathbf{L}^{0}(N,\tau)=L^{0}(N,\tau)$ as per Remark \ref{REM.Wstar_L2Red}, note Lemma \ref{LEM.Wstar_L2Red_III} lets us extend $N[1_{M}]=N\oplus\langle 1_{N}^{\perp}\rangle_{\mathbb{C}}$ to $L^{0}\lc{}N[1_{M}],\tau\rc{}=L^{0}(N,\tau)\oplus\langle 1_{N}^{\perp}\rangle_{\mathbb{C}}$ as per Equation \ref{EQ.SSEC.B_JFC_Compression_1} using direct sum of $^{*}$-algebras s.t.~integrability is preserved. Theorem \ref{THM.Wstar_L2Red} and its Corollary \ref{COR.Wstar_L2Red_III} ensure this extends to $L^{p}$-norms for all $p\in [1,\infty]$. We may therefore forget all a priori complications underlying Definition \ref{DFN.Wstar_L2Red_Compression}, treating $\comN=\comunit$ and $\comunitperp$ as in the bounded case. We make this explicit in Diagram \ref{EQ.SSEC.B_JFC_Compression_3}.\par


\pagebreak


\begin{dfn}
For all $x\in\mathbf{L}^{0}\lc{}N[1_{M}],\tau\rc$, set $\nu_{x}:=\nu_{\RE(x)}+i\nu_{\IM(x)}$.
\end{dfn}

\begin{lem}\label{LEM.Wstar_L2Red_III}
Let $N\subset (M,\tau)$. For all $x\in\mathbf{L}^{0}\lc{}N[1_{M}],\tau\rc$, we have

\begin{itemize}
\item[1)] $x=1_{N}x1_{N}+\nu_{x}1_{N}^{\perp}$,

\item[2)] $x\in\mathbf{L}^{0}(N,\tau)$ if and only if $\nu_{\RE(x)}=\nu_{\IM(x)}=0$,

\item[3)] $\tau\lc\absv{0.925}{x}\rc{}=\tau\lc\absv{0.925}{\comN x}\rc\in [0,\infty]$ if $x\in\mathbf{L}^{0}(N,\tau)$.
\end{itemize}
\end{lem}
\begin{proof}
If $N\subset M$ is a unital $W^{*}$-subalgebra, then we reduce to the non-unital case for $\nu_{x}=0$. We assume $N\subset M$ is a non-unital $W^{*}$-subalgebra without loss of generality. Let $x\in\mathbf{L}^{0}\lc{}N[1_{M}],\tau\rc$. We require $1)$ to show $2)$ and $3)$.\par
We show $1)$. As $\mathbf{L}^{0}\lc{}N[1_{M}],\tau\rc$ is a $^{*}$-subalgebra by $1)$ in Proposition \ref{PRP.Wstar_L2Red_I}, we assume $x$ is self-adjoint without loss of generality. We show Equation \ref{EQ.LEM.Wstar_L2Red_III_7} to get decomposition as per $1)$ by the spectral theorem. Using $2)$ in Lemma \ref{LEM.Wstar_L2Red_II} for the first and third identity, as well as $2)$ in Corollary \ref{COR.Wstar_L2Red_I} for the second one, we calculate

\begin{align*}
\textrm{com}_{\mathbf{L}^{2}(N,\tau)}\hspace{0.05cm} L_{x,M} & = \textrm{com}_{\mathbf{L}^{2}(N,\tau)}\hspace{0.05cm} \textrm{inc}_{2}^{\dagger}\lc{}L_{\comN x,N}\rc \phantom{\Bigg)} \\
& = \textrm{com}_{\mathbf{L}^{2}(N,\tau)}\hspace{0.05cm} \textrm{inc}_{2}^{\dagger}\lc{}L_{\comN 1_{N}x1_{N},N}\rc \phantom{\Bigg)} \\
& = \textrm{com}_{\mathbf{L}^{2}(N,\tau)}\hspace{0.05cm} L_{1_{N}x1_{N},M}. \phantom{\Bigg)}
\end{align*}

We show Equation \ref{EQ.LEM.Wstar_L2Red_III_7}. Let $Z\in\mathfrak{B}(\mathbb{R})$. Using the above calculation, Equation \ref{EQ.LEM.Wstar_L2Red_I_3} immediately implies

\begin{align}\label{EQ.LEM.Wstar_L2Red_III_1}
L_{N}\lc{}E_{x,N}(Z)\rc{}=L_{N}\lc{}E_{1_{N}x1_{N},N}(Z)\rc{}.
\end{align}

\noindent Note Equation \ref{EQ.SSEC.B_JFC_L2Red_3} ensures $E_{x,N}(Z),E_{1_{N}x1_{N},N}(Z)\in N$ by definition. Moreover, we know $L_{N}:N\longrightarrow\BII\lc{}L^{2}(N,\tau)\rc$ is faithful by $1)$ in Proposition \ref{PRP.Wstar_Trace_SF_Subalg}. Applying $L_{N}^{-1}$ to both sides of Equation \ref{EQ.LEM.Wstar_L2Red_III_1} yields

\begin{align}\label{EQ.LEM.Wstar_L2Red_III_2}
E_{x,N}(Z)=E_{1_{N}x1_{N},N}(Z).
\end{align}

\noindent Equation \ref{EQ.SSEC.B_JFC_L2Red_2} and Equation \ref{EQ.LEM.Wstar_L2Red_III_2} show we have decomposition

\begin{align}\label{EQ.LEM.Wstar_L2Red_III_3}
E_{x,M}(Z)=E_{1_{N}x1_{N},N}(Z)\oplus\nu_{x,N}1_{N}^{\perp}(Z)
\end{align}

\noindent w.r.t.~$N[1_{M}]=N\oplus\langle 1_{N}^{\perp}\rangle_{\mathbb{C}}$.\par


\pagebreak


Using $N1_{N}^{\perp}=1_{N}^{\perp}N=0$, we directly verify

\begin{align}\label{EQ.LEM.Wstar_L2Red_III_4}
R_{\pm i}\lc{}1_{N}x1_{N}+\nu_{x}1_{N}^{\perp}\rc{}=1_{N}\lc{}1_{N}x1_{N}\mp i1_{M}\rc^{-1}1_{N}+1_{N}^{\perp}\lc\nu_{x}1_{N}^{\perp}\mp i1_{M}\rc^{-1}1_{N}^{\perp}
\end{align}

\noindent in $L^{0}(M,\tau)$. Using $3)$ in Proposition \ref{PRP.Wstar_CLRA_FC} and bounded measurable functional calculus as per Definition \ref{DFN.Wstar_CLRA_FC_III} inside compression terms, Equation \ref{EQ.LEM.Wstar_L2Red_III_4} implies

\begin{align}\label{EQ.LEM.Wstar_L2Red_III_5}
E_{1_{N}x1_{N}+\nu_{x}1_{N}^{\perp},M}(Z)=E_{1_{N}x1_{N},N}(Z)\oplus 1_{N}^{\perp}E_{\nu_{x}1_{N}^{\perp},M}(Z)1_{N}^{\perp}
\end{align}

\noindent w.r.t.~$N[1_{M}]=N\oplus\langle 1_{N}^{\perp}\rangle_{\mathbb{C}}$. Note Equation \ref{EQ.SSEC.B_JFC_L2Red_3} ensures $E_{1_{N}x1_{N},N}(Z)=1_{N}E_{1_{N}x1_{N},M}(Z)1_{N}$ by definition. Lemma \ref{LEM.Wstar_CLRA_FC} shows elements in $W_{M}^{*}\lc\nu_{x}1_{N}\rc$ as per $3)$ in Proposition \ref{PRP.Wstar_CLRA_FC} are strong limits of finite polynomials with elements in $\lset\nu_{x}1_{N}^{\perp},1_{M}\rset$. Using the latter, get $\nu_{x,N}1_{N}^{\perp}(Z)+1_{N}=E_{\nu_{x}1_{N}^{\perp},M}(Z)$ since $E_{\nu_{x}1_{N}^{\perp},M}(\mathbb{R})=1_{M}$. This implies

\begin{align}\label{EQ.LEM.Wstar_L2Red_III_6}
\nu_{x,N}1_{N}^{\perp}(Z)=1_{N}^{\perp}E_{\nu_{x}1_{N}^{\perp},M}1_{N}^{\perp}(Z).
\end{align}

Equation \ref{EQ.LEM.Wstar_L2Red_III_6} shows the right-hand sides of Equation \ref{EQ.LEM.Wstar_L2Red_III_5} and Equation \ref{EQ.LEM.Wstar_L2Red_III_3} are identical in each case. For all $Z\in\mathfrak{B}(\mathbb{R})$, we therefore have

\begin{align}\label{EQ.LEM.Wstar_L2Red_III_7}
E_{x,M}(Z)=E_{1_{N}x1_{N}+\nu_{x}1_{N}^{\perp},M}(Z).
\end{align}

\noindent Using $2)$ in Lemma \ref{LEM.Wstar_CLRA_FC}, Equation \ref{EQ.LEM.Wstar_L2Red_III_7} shows the spectral theorem implies $1)$ since the given spectral measures coincide. Note $1)$ shows $2)$ at once.\par
We show $3)$. Equation \ref{EQ.LEM.Wstar_L2Red_I_5} shows $\absv{0.925}{\comN x}=\comN \absv{0.925}{x}$. We directly verify $\absv{1.15}{1_{N}x1_{N}}=1_{N}\absv{0.925}{x}1_{N}$. Thus $\absv{0.925}{x}\in\mathbf{L}^{0}(N,\tau)$ if $x\in \mathbf{L}^{0}(N,\tau)$, hence we assume $x\in\mathbf{L}^{0}(N,\tau)$ is positive without loss of generality. Then $2)$ implies $\nu_{x}=0$. Note the infimum in Equation \ref{EQ.SSEC.B_SMO_NCI_5} runs over all $\lambda>0$. Equation \ref{EQ.SSEC.B_JFC_L2Red_2} and $\nu_{x}=0$ therefore imply the generalised singular number of $x$ is given by

\begin{align*}
\mu_{t}(x)=
\begin{cases}
\mu_{t-\tau(1_{N}^{\perp})}\lc\comN x\rc{} & \If\ t\geq \tau(1_{N}^{\perp}), \\
0 & \Else.
\end{cases}
\end{align*}

\noindent We use the explicit expression above to calculate

\begin{align}\label{EQ.LEM.Wstar_L2Red_III_8}
\tau(x)=\int_{0}^{\infty}\mu_{t}(x)dt=\int_{\tau(1_{N}^{\perp})}^{\infty}\mu_{t-\tau(1_{N}^{\perp})}\lc\comN x\rc{}dt=\int_{0}^{\infty}\mu_{t}\lc\comN x\rc{}dt=\tau\lc\comN x\rc{}.
\end{align}

\noindent Following Remark \ref{REM.Wstar_SMO_Trace}, Equation \ref{EQ.LEM.Wstar_L2Red_III_8} shows $\tau\lc\absv{0.925}{x}\rc{}=\tau\lc\absv{0.925}{\comN x}\rc$. Get $3)$.
\end{proof}


\pagebreak


\begin{cor}\label{COR.Wstar_L2Red_II}
Let $N\subset (M,\tau)$. For all $x\in\mathbf{L}^{0}(N,\tau)$, we have

\begin{itemize}
\item[1)] $x\in L^{1}(M,\tau)$ if and only if $\comN x\in L^{1}(N,\tau)$,

\item[2)] $\|x\|_{1}=\dblv{}\comN x\dblv_{1}$ and $\tau(x)=\tau\lc\comN x\rc$ if $x\in L^{1}(M,\tau)$.
\end{itemize}
\end{cor}
\begin{proof}
Note $3)$ in Lemma \ref{LEM.Wstar_L2Red_III} shows $1)$. If we furthermore extend Equation \ref{EQ.LEM.Wstar_L2Red_III_8} to all $x\in\mathbf{L}^{0}(N,\tau)\cap L^{1}(M,\tau)$, then $2)$ follows. Equation \ref{EQ.LEM.Wstar_L2Red_I_4} shows $\comN:\mathbf{L}^{0}(N,\tau)\longrightarrow L^{0}(N,\tau)$ is linear and positivity-preserving by $1)$ in Corollary \ref{COR.Unbd_Twist_RP} and Proposition \ref{PRP.Compression_Concrete_RP}. For all $\lambda\in\mathbb{R}$, $\max\{\lambda,0\}=\frac{1}{2}\lc\lambda+\absv{1}{\lambda}\rc$ and $\min\{\lambda,0\}=\frac{1}{2}\lc\lambda-\absv{1}{\lambda}\rc$. We use decomposition as per Proposition \ref{PRP.Wstar_NCI_IV} and thereby see linearity, positivity-preservation and Equation \ref{EQ.LEM.Wstar_L2Red_I_5} extend Equation \ref{EQ.LEM.Wstar_L2Red_III_8} to all $x\in\mathbf{L}^{0}(N,\tau)\cap L^{1}(M,\tau)$.
\end{proof}

\begin{lem}\label{LEM.Wstar_L2Red_IV}
If $N\subset (M,\tau)$, then

\begin{itemize}
\item[1)] $\comN:\mathbf{L}^{0}\lc{}N[1_{M}],\tau\rc\longrightarrow L^{0}(N,\tau)$ is a surjective $^{*}$-homomorphism,

\item[2)] $\comN:\mathbf{L}^{0}(N,\tau)\longrightarrow L^{0}(N,\tau)$ is a $^{*}$-isomorphism,

\item[3)] we have commutative diagram

\begin{equation}\label{EQ.LEM.Wstar_L2Red_IV_1}
\begin{tikzcd}
\mathbf{L}^{0}(N,\tau)\arrow[rr,hook]\arrow[dd,"\comN"] & & \mathbf{L}^{0}\lc{}N[1_{M}],\tau\rc\arrow[rr,"L_{M}"]\arrow[dd,"\comN"] & & \UBII\lc{}L^{2}(M,\tau)\rc\arrow[dd,"\inc_{2}^{-\dagger}\circ\hspace{0.0675cm} \com_{\mathbf{L}^{2}(N,\tau)}"] \\
& & & & \\
L^{0}(N,\tau)\arrow[rr,"\id_{L^{0}(N,\tau)}"] & & L^{0}(N,\tau)\arrow[rr,"L_{N}"] & & \UBII\lc{}L^{2}(N,\tau)\rc{}
\end{tikzcd}
\end{equation}

\begin{reapply}
\end{reapply}

\noindent s.t.~horizontal maps are normal unital injective $^{*}$-homomorphisms and vertical ones are positivity-preserving linear surjections,

\item[3)] $\mathbf{L}^{2}(N,\tau)=\comN^{-1}\lc{}L^{2}(N,\tau)\rc$ and $\comN^{-1}:L^{2}(N,\tau)\longrightarrow\mathbf{L}^{2}(N,\tau)$ is an isometric isomorphism of Hilbert spaces restricting to the identity on $N$.
\end{itemize}
\end{lem}
\begin{proof}
Equation \ref{EQ.LEM.Wstar_L2Red_I_4} shows Diagram \ref{EQ.LEM.Wstar_L2Red_IV_1} for self-adjoint elements. Said equation also shows $\comN:\mathbf{L}^{0}\lc{}N[1_{M}],\tau\rc\longrightarrow L^{0}(N,\tau)$ is linear. Twisting and concrete compression maps are linear and positivity-, hence order-preserving by $1)$ in Corollary \ref{COR.Unbd_Twist_RP} and Proposition \ref{PRP.Compression_Concrete_RP}. Get Diagram \ref{EQ.LEM.Wstar_L2Red_IV_1} from the case of self-adjoint elements.\par
Thus $\comN:\mathbf{L}^{0}\lc{}N[1_{M}],\tau\rc\longrightarrow L^{0}(N,\tau)$ is a positivity-, hence order-preserving linear map. In particular, $\comN$ commutes with algebra involution. If we have

\begin{align}\label{EQ.LEM.Wstar_L2Red_IV_2}
\textrm{com}_{\mathbf{L}^{2}(N,\tau)}\hspace{0.05cm} L_{xy,M}=\overline{\textrm{com}_{\mathbf{L}^{2}(N,\tau)}\hspace{0.05cm} L_{x,M}\cdot \textrm{com}_{\mathbf{L}^{2}(N,\tau)}\hspace{0.05cm} L_{y,M}}
\end{align}

\noindent for all $x,y\in\mathbf{L}^{0}\lc{}N[1_{M}],\tau\rc$, then we see $3.3)$ in Proposition \ref{PRP.Unbd_Twist} and Diagram \ref{EQ.LEM.Wstar_L2Red_IV_2} imply $\comN$ is a $^{*}$-homomorphism.\par
For all $y\in\mathbf{L}^{0}\lc{}N[1_{M}],\tau\rc$, we know $\mathbf{L}^{2}(N,\tau)$-reducibility by $2)$ in Proposition \ref{PRP.Wstar_L2Red_II} and use $\pi_{\mathbf{L}^{2}(N,\tau)}L_{y,M}\subset L_{y,M}\pi_{\mathbf{L}^{2}(N,\tau)}$ as per Equation \ref{EQ.DFN.Reducible_1} to get

\begin{align}\label{EQ.LEM.Wstar_L2Red_IV_3}
\textrm{com}_{\mathbf{L}^{2}(N,\tau)}\hspace{0.05cm} L_{y,M}=\pi_{\mathbf{L}^{2}(N,\tau)}\cdot L_{y,M}\cdot \pi_{\mathbf{L}^{2}(N,\tau)}=L_{y,M}\cdot \pi_{\mathbf{L}^{2}(N,\tau)}.
\end{align}

\noindent Note left-~and right-multiplication with $\pi_{\mathbf{L}^{2}(N,\tau)}$ commute with taking closure. Using Corollary \ref{COR.Wstar_CLRA_III} and Equation \ref{EQ.LEM.Wstar_L2Red_IV_3}, we directly verify Equation \ref{EQ.LEM.Wstar_L2Red_IV_2}. Get $1)$ and $3)$.\par
We show $2)$. We must show injectivity. It suffices to consider self-adjoint elements. By Lemma \ref{LEM.Wstar_L2Red_II} and $2)$ in Lemma \ref{LEM.Wstar_L2Red_III}, we assume $N\subset M$ is a unital $W^{*}$-subalgebra without loss of generality. For all $x\in\mathbf{L}^{0}(N,\tau)$, get $E_{x,M}=E_{x,N}$. The spectral theorem then shows $L_{\comN x,N}$ determines $L_{x,M}$ uniquely in each case. Injectivity as required for $2)$ holds. We show $4)$. For all $x\in\mathbf{L}^{0}(N,\tau)$, $\comN \absv{0.925}{x}^{2}=\absv{0.925}{\comN x}^{2}$ because $\comN$ is a positivity-preserving $^{*}$-homomorphism. Get $4)$ by $2)$ and Corollary \ref{COR.Wstar_L2Red_II}.
\end{proof}

\begin{rem}
Note $\inc_{2}=\comN^{-1}$ on $L^{2}(N,\tau)$ since both are isometries restricting to the identical map on a dense subset. 
\end{rem}

We have $\mathbf{L}^{2}(N,\tau)=\comN^{-1}\lc{}L^{2}(N,\tau)\rc$ and identify along $\inc_{2}$. This is subsumed by the general convention fixed in Remark \ref{REM.Wstar_L2Red}.

\begin{thm}\label{THM.Wstar_L2Red}
Let $(M,\tau)$ be a tracial $W^{*}$-algebra. If $N\subset (M,\tau)$, then

\begin{itemize}
\item[1)] $\comN^{-1}:L^{0}(N,\tau)\longrightarrow\mathbf{L}^{0}(N,\tau)$ is a $^{*}$-isomorphism,

\item[2)] we have commutative diagram

\begin{equation}\label{EQ.THM.Wstar_L2Red_1}
\begin{tikzcd}
\mathbf{L}^{0}(N,\tau)\arrow[rr,hook]\arrow[dd,"\comN"] & & \mathbf{L}^{0}\lc{}N[1_{M}],\tau\rc\arrow[rr,"L_{M}"]\arrow[dd,"\comN"] & & \UBII\lc{}L^{2}(M,\tau)\rc\arrow[dd,"\com_{L^{2}(N,\tau)}"] \\
& & & & \\
L^{0}(N,\tau)\arrow[rr,"\id_{L^{0}(N,\tau)}"] & & L^{0}(N,\tau)\arrow[rr,"L_{N}"] & & \UBII\lc{}L^{2}(N,\tau)\rc{}
\end{tikzcd}
\end{equation}

\begin{reapply}
\end{reapply}

\noindent s.t.~horizontal maps are normal unital injective $^{*}$-homomorphisms and vertical ones are positivity-preserving linear surjections,

\item[3)] for all $p\in [1,\infty]$, $\mathbf{L}^{p}(N,\tau)=\comN^{-1}\lc{}L^{p}(N,\tau)\rc$ and $\comN^{-1}:L^{p}(N,\tau)\longrightarrow\mathbf{L}^{p}(N,\tau)$ is an isometric isomorphism of Banach spaces restricting to the identity on $N$.
\end{itemize}
\end{thm}
\begin{proof}
We identify along $\inc_{2}$. Lemma \ref{LEM.Wstar_L2Red_IV} implies $1)$ and $2)$ at once. We show $3)$. For this, we instead require the following. Let $x\in\mathbf{L}^{0}\lc{}N[1_{M}],\tau\rc$. If $\alpha\in\mathbb{Q}$, then we know

\begin{align}\label{EQ.THM.Wstar_L2Red_2}
\comN \absv{0.925}{x}^{\alpha}=\absv{0.925}{\comN x}^{\alpha}    
\end{align}

\noindent because $\comN$ is a positivity-preserving $^{*}$-homomorphism by $1)$ in Lemma \ref{LEM.Wstar_L2Red_IV}. We show Equation \ref{EQ.THM.Wstar_L2Red_2} holds for all $\alpha\geq 0$.\par
Let $\beta\geq 0$. For all $\lambda\geq 0$, set $g_{\beta}(\lambda):=\lambda^{\beta}$. Get $R_{i}\lc{}g_{\beta}\rc\in C_{0}\lc{}[0,\infty)\rc$ and $R_{i}\lc\absv{0.925}{x}^{\beta}\rc\in\mathbf{L}^{0}(N,\tau)$. Since $\comN$ is a positivity-preserving $^{*}$-homomorphism, get

\begin{align}\label{EQ.THM.Wstar_L2Red_3}
\comN R_{i}\lc\absv{0.925}{x}^{\beta}\rc{}=R_{i}\lc\comN \absv{0.925}{x}^{\beta}\rc{}.
\end{align}
    
\noindent Let $\alpha\geq 0$ and $\lset\alpha_{n}\rset_{n\in\mathbb{N}}\subset\mathbb{Q}\cap (0,\infty)$ s.t.~$\alpha=\lim_{n\in\mathbb{N}}\alpha_{n}$. We thereby have uniformly bounded pointwise limit $R_{i}\lc{}g_{\alpha}\rc{}=\lim_{n\in\mathbb{N}}R_{i}\lc{}g_{\alpha_{n}}\rc$ in $C_{0}\lc{}[0,\infty)\rc$ and calculate

\begin{align}\label{EQ.THM.Wstar_L2Red_4}
R_{i}\lc\absv{0.925}{x}^{\alpha}\rc{}=\s\textrm{-}\lim_{n\in\mathbb{N}}\hspace{0.025cm} R_{i}\lc\absv{0.925}{x}^{\alpha_{n}}\rc{},\ R_{i}\lc\absv{0.925}{\comN x}^{\alpha}\rc{}=\s\textrm{-}\lim_{n\in\mathbb{N}}\hspace{0.025cm} R_{i}\lc\absv{0.925}{\comN x}^{\alpha_{n}}\rc{}.
\end{align}

\noindent Note $2)$ in Corollary \ref{COR.Wstar_L2Red_I} shows $\comN=\comunit$ on $N[1_{M}]$. Thus $\comN$ restricted to $N[1_{M}]$ is a completely positive normal bounded linear map by Proposition \ref{PRP.Compression_Abstract_Bd}, hence bounded strongly continuous by Proposition \ref{PRP.Wstar_Normal}. Using bounded strong continuity in each case, Equation \ref{EQ.THM.Wstar_L2Red_2}, Equation \ref{EQ.THM.Wstar_L2Red_3} and Equation \ref{EQ.THM.Wstar_L2Red_4} let us calculate

\begin{align*}
\comN R_{i}\lc\absv{0.925}{x}^{\alpha}\rc{} & = \s\textrm{-}\lim_{n\in\mathbb{N}}\hspace{0.025cm} \comN R_{i}\lc\absv{0.925}{x}^{\alpha_{n}}\rc \phantom{\bigg)} \\
& = \s\textrm{-}\lim_{n\in\mathbb{N}}\hspace{0.025cm} R_{i}\lc\comN \absv{0.925}{x}^{\alpha_{n}}\rc \phantom{\bigg)} \\
& = \s\textrm{-}\lim_{n\in\mathbb{N}}\hspace{0.025cm} R_{i}\lc\absv{0.925}{\comN x}^{\alpha_{n}}\rc \phantom{\bigg)} \\
& = R_{i}\lc\absv{0.925}{\comN x}^{\alpha}\rc{}. \phantom{\bigg)}
\end{align*}

The above calculation shows

\begin{align}\label{EQ.THM.Wstar_L2Red_5}
\comN R_{i}\lc\absv{0.925}{x}^{\alpha}\rc{}=R_{i}\lc\absv{0.925}{\comN x}^{\alpha}\rc{}.
\end{align}

\noindent Since $\comN 1_{M}=\comunit 1_{M}=1_{N}$, get $\comN\ \lc\absv{0.925}{x}^{\alpha}-i1_{M}\rc{}=\comN \absv{0.925}{x}^{\alpha}-i1_{N}$. Using the latter as first and Equation \ref{EQ.THM.Wstar_L2Red_5} for the second identity below, we calculate 

\begin{align}\label{EQ.THM.Wstar_L2Red_6}
\comN \absv{0.925}{x}^{\alpha}-i1_{N}=\lc\comN R_{i}\lc\absv{0.925}{x}^{\alpha}\rc\rc^{-1}=\absv{0.925}{\comN x}^{\alpha}-i1_{N}.
\end{align}

\noindent Equation \ref{EQ.THM.Wstar_L2Red_6} implies Equation \ref{EQ.THM.Wstar_L2Red_2} for all $\alpha\geq 0$.\par
Let $p\in [1,\infty)$. For all $x\in\mathbf{L}^{0}(N,\tau)$, Equation \ref{EQ.THM.Wstar_L2Red_2} shows $\comN \absv{0.925}{x}^{p}=\absv{0.925}{\comN x}^{p}$. If $p<\infty$, then we know the latter identity implies $3)$ by $1)$ and Corollary \ref{COR.Wstar_L2Red_II}. The case of $p=\infty$ is clear. Get $3)$.
\end{proof}

\begin{rem}\label{REM.Wstar_L2Red}
Let $(M,\tau)$ be a tracial $W^{*}$-algebra. If $N\subset (M,\tau)$, then Theorem \ref{THM.Wstar_L2Red} ensures we identify along $\comN$ without loss of generality. Note identification preserves $^{*}$-algebra structure, positivity and noncommutative $L^{p}$-norms. Corollary \ref{COR.Wstar_L2Red_II} ensures identification further preserves trace.
\end{rem}


\subsubsection*{Compression maps on spaces of measurable operators}

If $(M,\tau)$ is a tracial $W^{*}$-algebra and $p\in M$ is a projection, then $M[p]\subset (M,\tau)$ by $2)$ in Proposition \ref{PRP.Wstar_Trace_NCE_II} and $L^{0}\lc{}M[p],\tau\rc\subset L^{0}(M,\tau)$ by Theorem \ref{THM.Wstar_L2Red}.

\begin{prp}\label{PRP.Wstar_L2Red_III}
Let $(M,\tau)$ be a tracial $W^{*}$-algebra. For all projections $p\in M$ and $q\in [1,\infty]$, we have $L^{0}\lc{}M[p],\tau\rc{}=pL^{0}(M,\tau)p$ and $L^{q}\lc{}M[p],\tau\rc{}=pL^{q}(M,\tau)p$.
\end{prp}
\begin{proof}
Multiplication in $L^{0}(M,\tau)$ is continuous in measure topology on bounded subsets \lc{}cf.~Theorem IX.2.2 in \cite{BK.Tak.2003.OpAlg_II} and \cite{ART.Nel.1974.Wstar_Integration}\rc{}. Using the latter, $2)$ in Proposition \ref{PRP.Wstar_L2Red_I} and $M[p]=pMp$, we have $L^{0}\lc{}M[p],\tau\rc{}=\overline{M[p]}=pL^{0}(M,\tau)p$ by approximating in measure topology. For all $q\in [1,\infty]$, get $L^{q}\lc{}M[p],\tau\rc{}=pL^{q}(M,\tau)p$ by $3)$ in Theorem \ref{THM.Wstar_L2Red}.
\end{proof}

\begin{dfn}\label{DFN.Compression_Abstract}
Let $(M,\tau)$ be a tracial $W^{*}$-algebra. For all projections $p\in M$, we define the compression map $\comp:L^{0}(M,\tau)\longrightarrow L^{0}\lc{}M[p],\tau\rc$ by setting

\begin{align}\label{EQ.DFN.Compression_Abstract_1}
\comp x:=pxp 
\end{align}

\noindent for all $x\in L^{0}(M,\tau)$.
\end{dfn}

\begin{cor}\label{COR.Wstar_L2Red_III}
If $N\subset (M,\tau)$, then

\begin{itemize}
\item[1)] $\comN=\comunit$ on $L^{0}\lc{}N[1_{M}],\tau\rc$,

\item[2)] we have commutative diagrams

\begin{equation}\label{EQ.COR.Wstar_L2Red_III_1}
\begin{tikzcd}
L^{0}(N,\tau)\arrow[rr,hook]\arrow[dd,"\id_{L^{0}(N,\tau)}"] & & L^{0}\lc{}N[1_{M}],\tau\rc\arrow[rr,"L_{M}"]\arrow[dd,"\comunit"] & & \UBII\lc{}L^{2}(M,\tau)\rc\arrow[dd,"\com_{L^{2}(N,\tau)}"] \\
& & & & \\
L^{0}(N,\tau)\arrow[rr,"\id_{L^{0}(N,\tau)}"] & & L^{0}(N,\tau)\arrow[rr,"L_{N}"] & & \UBII\lc{}L^{2}(N,\tau)\rc{}
\end{tikzcd}
\end{equation}

\begin{reapply}
\end{reapply}

\begin{equation}\label{EQ.COR.Wstar_L2Red_III_2}
\begin{tikzcd}
L^{0}(N,\tau)^{\op}\arrow[rr,hook]\arrow[dd,"\id_{L^{0}(N,\tau)}"] & & L^{0}\lc{}N[1_{M}],\tau\rc^{\op}\arrow[rr,"R_{M}"]\arrow[dd,"\comunit"] & & \UBII\lc{}L^{2}(M,\tau)\rc\arrow[dd,"\com_{L^{2}(N,\tau)}"] \\
& & & & \\
L^{0}(N,\tau)^{\op}\arrow[rr,"\id_{L^{0}(N,\tau)}"] & & L^{0}(N,\tau)^{\op}\arrow[rr,"R_{N}"] & & \UBII\lc{}L^{2}(N,\tau)\rc{}
\end{tikzcd}
\end{equation}

\begin{reapply}
\end{reapply}

\noindent s.t.~horizontal maps are normal unital injective $^{*}$-homomorphisms and vertical ones are positivity-preserving linear surjections.
\end{itemize}
\end{cor}
\begin{proof}
Get $1)$ by Corollary \ref{COR.Wstar_L2Red_I}. Thus Diagram \ref{EQ.COR.Wstar_L2Red_III_1} is $2)$ in Theorem \ref{THM.Wstar_L2Red}, hence Diagram \ref{EQ.COR.Wstar_L2Red_III_2} follows from Diagram \ref{EQ.COR.Wstar_L2Red_III_1} by $2)$ in Corollary \ref{COR.Wstar_CLRA_II}. 
\end{proof}


\subsection{Compressed pulled-back joint functional calculus}\label{SSEC.B_JFC_Compression}

We prove Theorem \ref{THM.JFC_Compression} and give compressed pulled-back joint functional calculus of self-adjoint measurable operators in Definition \ref{DFN.JFC_Compression_IV}. In Subsection \ref{SSEC.NCDS_AF_FC}, we apply Theorem \ref{THM.JFC_Compression} and its corollaries.


\subsubsection*{Change of canonical left-~and right-actions}

We express change of canonical left-~and right-actions as abstract compression maps. Let $(M,\tau)$ be a tracial $W^{*}$-algebra and $N\subset (M,\tau)$. We identify as per Remark \ref{REM.Wstar_L2Red}. For all $x\in L^{0}\lc{}N[1_{M}],\tau\rc$, note $1)$ in Lemma \ref{LEM.Wstar_L2Red_III} shows we have decomposition

\begin{align}\label{EQ.SSEC.B_JFC_Compression_1}
x=\comunit x \oplus \nu_{x}1_{N}^{\perp}
\end{align}

\noindent extending $N[1_{M}]=N\oplus\langle 1_{N}^{\perp}\rangle_{\mathbb{C}}$ to $L^{0}\lc{}N[1_{M}],\tau\rc{}=L^{0}(N,\tau)\oplus\langle 1_{N}^{\perp}\rangle_{\mathbb{C}}$ using the direct sum of $^{*}$-algebras induced by canonical inclusion in $L^{0}(M,\tau)$. Mapping as per Equation \ref{EQ.SSEC.B_JFC_Compression_1} using left diagram in Diagram \ref{EQ.COR.Wstar_L2Red_III_1}, we indeed have $L^{0}\lc{}N[1_{M}],\tau\rc{}=L^{0}(N,\tau)\oplus\langle 1_{N}^{\perp}\rangle_{\mathbb{C}}$ and commutative diagram

\smallskip

\begin{equation}\label{EQ.SSEC.B_JFC_Compression_3}
\begin{tikzcd}
L^{0}(N,\tau)\arrow[rr,hook]\arrow[rrrr, bend right=33, "\id_{L^{0}(N,\tau)}"] & & L^{0}(N,\tau)\oplus\langle 1_{N}^{\perp}\rangle_{\mathbb{C}}\arrow[rr,"\comunit"] & & L^{0}(N,\tau)
\end{tikzcd}   
\end{equation}

\medskip

\noindent of $^{*}$-homomorphisms. Diagram \ref{EQ.SSEC.B_JFC_Compression_3} extends Diagram \ref{EQ.SSEC.A_Maps_Compression_1}.

\begin{lem}\label{LEM.Wstar_Compression_Preservation}
If $x\in L^{0}\lc{}N[1_{M}],\tau\rc_{h}$, then we have

\begin{itemize}
\item[1)] $\specM x=\specN\comunit x\cup\{\nu_{x}\}$ and $\mathcal{N}\lc{}E_{x,M}\rc{}=\mathcal{N}\big(E_{\comunit x,N}\big)\cap\mathcal{N}(\nu_{x})$,

\item[2)] normal unital surjective $^{*}$-homomorphism $\comunit :W_{M}^{*}(x)\longrightarrow W_{N}^{*}\lc\comunit x\rc$ s.t.~

\begin{align}\label{EQ.LEM.Wstar_Compression_Preservation_1}
\comunit\Gamma_{x,M}(g)=\Gamma_{\comunit x,N}(g)   
\end{align}

\begin{reapply}
\end{reapply}

\noindent for all $g\in L^{\infty}\lc\specM x,dE_{x,M}\rc$,

\item[3)] commutative diagram of normal unital surjective $^{*}$-homomorphisms

\begin{equation}\label{EQ.LEM.Wstar_Compression_Preservation_2}
\begin{tikzcd}
L^{\infty}\lc\specM x,dE_{x,M}\rc\arrow[rr,"\Gamma_{x,M}"]\arrow[dd,"\res"] & & W_{M}^{*}(x)\arrow[dd,"\comunit"] \\
& & \\
L^{\infty}\lc\specN\comunit x,dE_{\comunit x,N}\rc\arrow[rr,"\Gamma_{\comunit x,N}"] & & W_{N}^{*}\lc\comunit x\rc{}
\end{tikzcd}
\end{equation}

\begin{reapply}
\end{reapply}

\noindent with $\res$ the restriction map given by $\specN\comunit x\subset\specM x$.
\end{itemize}
\end{lem}


\pagebreak


\begin{proof}
Let $x\in L^{0}\lc{}N[1_{M}],\tau\rc_{h}$. Decomposing $L^{0}\lc{}N[1_{M}],\tau\rc{}=L^{0}(N,\tau)\oplus\langle 1_{N}^{\perp}\rangle_{\mathbb{C}}$, we directly verify $1)$. We show $2)$ and $3)$ using Lemma \ref{LEM.Wstar_CLRA_FC}. Note $1)$ in Corollary \ref{COR.Wstar_L2Red_III} shows $\comN=\comunit$ on $L^{0}\lc{}N[1_{M}],\tau\rc$. Since we identify as per Remark \ref{REM.Wstar_L2Red}, we know $2)$ in Lemma \ref{LEM.Wstar_L2Red_II} therefore implies

\begin{align}\label{EQ.LEM.Wstar_Compression_Preservation_3}
\restr{0.875}{L_{x,M}}{L^{2}(N,\tau)}=\textrm{com}_{L^{2}(N,\tau)}\hspace{0.05cm} L_{x,M}=L_{\comunit x,N}.
\end{align}

\noindent Up to representation under canonical left-actions, Equation \ref{EQ.LEM.Wstar_Compression_Preservation_3} shows $2)$ and $3)$ in Lemma \ref{LEM.Wstar_CLRA_FC} are $2)$ and $3)$ as claimed. We therefore invert canonical left-actions and conclude by $2)$ in Lemma \ref{LEM.Wstar_CLRA_FC}.
\end{proof}

\begin{rem}
Theorem \ref{THM.Wstar_L2Red} and Corollary \ref{COR.Wstar_L2Red_III} show Lemma \ref{LEM.Wstar_Compression_Preservation} subsumes Corollary \ref{COR.Cstar_FC_I} and Lemma \ref{LEM.Compression_Preservation_I} for canonical left-~and right-actions of $L^{2}$-reducible measurable operators. Choice of unit only involves values at zero.
\end{rem}

\begin{cor}\label{COR.Wstar_Compression_Preservation_I}
If $x\in L^{0}(N,\tau)_{h}$, then we have

\begin{itemize}
\item[1)] $\specM x=\specN x\cup\lset{}0\rset$ and $\mathcal{N}\lc{}E_{x,M}\rc\subset\mathcal{N}\lc{}E_{x,N}\rc$,

\item[2)] $\restr{0.925}{g\lc{}L_{x,M}\rc{}}{L^{2}(N,\tau)}=g\lc{}L_{x,N}\rc$, $\restr{0.925}{g\lc{}R_{x,M}\rc{}}{L^{2}(N,\tau)}=g\lc{}R_{x,N}\rc$ for all $g\in L^{\infty}\lc\specM x,dE_{x,M}\rc$,

\item[3)] commutative diagrams of normal unital surjective $^{*}$-homomorphisms

\begin{equation}\label{EQ.COR.Wstar_Compression_Preservation_I_1}
\begin{tikzcd}
L^{\infty}\lc\specM x,dE_{x,M}\rc\arrow[rr,"\Gamma_{x,M}"]\arrow[dd,"\res"] & & W_{M}^{*}(x)\arrow[r,hook]\arrow[dd,"\comunit"] & M\arrow[r,"L_{M}"] & \BII\lc{}L^{2}(M,\tau)\rc\arrow[dd,"\com_{L^{2}(N,\tau)}"] \\
& & & & \\
L^{\infty}\lc\specN x,dE_{x,N}\rc\arrow[rr,"\Gamma_{x,N}"] & & W_{N}^{*}(x)\arrow[r,hook] & N\arrow[r,"L_{N}"] & \BII\lc{}L^{2}(N,\tau)\rc{}
\end{tikzcd}
\end{equation}

\begin{reapply}
\end{reapply}

\begin{equation}\label{EQ.COR.Wstar_Compression_Preservation_I_2}
\begin{tikzcd}
\hspace{-0.1125cm} L^{\infty}\lc\specM x,dE_{x,M}\rc\arrow[rr,"\Gamma_{x,M}"]\arrow[dd,"\res"] & & W_{M}^{*}(x)\arrow[r,hook]\arrow[dd,"\comunit"] & M^{\op}\arrow[r,"R_{M}"] & \BII\lc{}L^{2}(M,\tau)\rc\arrow[dd,"\com_{L^{2}(N,\tau)}"] \\
& & & & \\
\hspace{-0.1125cm} L^{\infty}\lc\specN x,dE_{x,N}\rc\arrow[rr,"\Gamma_{x,N}"] & & W_{N}^{*}(x)\arrow[r,hook] & N^{\op}\arrow[r,"R_{N}"] & \BII\lc{}L^{2}(N,\tau)\rc{}
\end{tikzcd}
\end{equation}

\begin{reapply}
\end{reapply}

\noindent with $\res$ the restriction map given by $\specN x\subset\specM x$.
\end{itemize}
\end{cor}
\begin{proof}
Get $1)$ by $2)$ in Lemma \ref{LEM.Wstar_L2Red_III} and $1)$ in Lemma \ref{LEM.Wstar_Compression_Preservation}. Diagram \ref{EQ.LEM.Wstar_Compression_Preservation_2} is the left diagram in both Diagram \ref{EQ.COR.Wstar_Compression_Preservation_I_1} and Diagram \ref{EQ.COR.Wstar_Compression_Preservation_I_2}. Diagram \ref{EQ.COR.Wstar_L2Red_III_1} and Diagram \ref{EQ.COR.Wstar_L2Red_III_2} yield right diagrams by restriction to bounded operators. This shows both $2)$ and $3)$.
\end{proof}

\begin{cor}\label{COR.Wstar_Compression_Preservation_II}
If $x\in L^{0}(N,\tau)_{h}$ and $g\in L^{\infty}\lc\specM x,dE_{x,M}\rc$ s.t.~$g(0)=0$, then

\begin{align}\label{EQ.COR.Wstar_Compression_Preservation_II_1}
\Gamma_{x,M}(g)=\Gamma_{x,N}(g).   
\end{align}
\end{cor}
\begin{proof}
Note $N\subset N[1_{M}]\subset (M,\tau)$ by $2)$ in Proposition \ref{PRP.Wstar_Trace_NCE_II}. By Theorem \ref{THM.Wstar_L2Red}, we assume $M=N[1_{M}]$ without loss of generality. Let $x\in L^{0}(N,\tau)_{h}$. We know $\nu_{x}=0$ by $2)$ in Lemma \ref{LEM.Wstar_L2Red_III}. Let $g\in L^{\infty}\lc\specM x,dE_{x,M}\rc$ s.t.~$g(0)=0$. Equation \ref{EQ.LEM.Wstar_Compression_Preservation_1} for $N\subset \lc{}N[1_{M}],\tau\rc$ and $\langle 1_{N}^{\perp}\rangle_{\mathbb{C}}\subset \lc{}N[1_{M}],\tau\rc$ each lets us calculate

\begin{align}\label{EQ.COR.Wstar_Compression_Preservation_II_2}
\Gamma_{x,M}(g)=\comunit\Gamma_{x,M}(g)+\comunitperp\Gamma_{x,M}(g)=\Gamma_{x,N}(g)+\Gamma_{0,\langle 1_{N}^{\perp}\rangle_{\mathbb{C}}}(g).
\end{align}

\noindent Get $\Gamma_{0,\langle 1_{N}^{\perp}\rangle_{\mathbb{C}}}(g)=g(0)=0$ by hypothesis. Equation \ref{EQ.COR.Wstar_Compression_Preservation_II_2} shows Equation \ref{EQ.COR.Wstar_Compression_Preservation_II_1}.
\end{proof}

\begin{cor}\label{COR.Wstar_Compression_Preservation_III}
If $x,y\in L^{0}(N,\tau)_{h}$, then we have

\begin{itemize}
\item[1)] $\specM x\times y=\lc\specN x\cup\lset{}0\rset\rc\times\lc\specN x\cup\lset{}0\rset\rc$ and $\mathcal{N}\lc{}E_{x,y,M}\rc\subset\mathcal{N}\lc{}E_{x,y,N}\rc$,

\item[2)] $\restr{0.925}{g\lc{}L_{x,M},R_{y,M}\rc{}}{L^{2}(N,\tau)}=g\lc{}L_{x,N},R_{y,N}\rc$ for all $g\in L^{\infty}\lc\specM x\times y,dE_{x,y,M}\rc$,

\item[3)] commutative diagram of normal unital surjective $^{*}$-homomorphisms

\begin{equation}\label{EQ.COR.Wstar_Compression_Preservation_III_1}
\begin{tikzcd}
L^{\infty}\lc\specM x\times y,dE_{x,y,M}\rc\arrow[rr,"\Gamma_{x,y,M}"]\arrow[dd,"\res"] & & W_{M}^{*}(x,y)\arrow[dd,"\comunitunit"]\arrow[rr,"L_{M}\otimes R_{M}"] & & \BII\lc{}L^{2}(M,\tau)\rc\arrow[dd,swap,"\com_{L^{2}(N,\tau)}"] \\
& & & & \\
L^{\infty}\lc\specN x\times y,dE_{x,y,N}\rc\arrow[rr,"\Gamma_{x,y,N}"] & & W_{N}^{*}(x,y)\arrow[rr,"L_{N}\otimes R_{N}"] & & \BII\lc{}L^{2}(N,\tau)\rc{}
\end{tikzcd}
\end{equation}

\begin{reapply}
\end{reapply}

\noindent with $\res$ the restriction map given by $\specN x\times y\subset\specM x\times y$.
\end{itemize}
\end{cor}
\begin{proof}
Note $2)$ Lemma \ref{LEM.Wstar_CLRA_JFC} and $1)$ in Corollary \ref{COR.Wstar_Compression_Preservation_I} show Lemma \ref{LEM.Compression_Preservation_II} applies to the outer diagram in Diagram \ref{EQ.COR.Wstar_Compression_Preservation_III_1}. Said lemma therefore shows $1)$, $2)$ and the outer diagram. Normality of all maps involved reduces the left diagram in Diagram \ref{EQ.COR.Wstar_Compression_Preservation_III_1} to elementary tensors. Apply $3)$ in Corollary \ref{COR.Wstar_Compression_Preservation_III}.
\end{proof}


\subsubsection*{The compression theorem}

Let $(M,\tau)$ be a tracial $W^{*}$-algebra and $H$ a Hilbert space. Let $N\subset (M,\tau)$ and $V\subset H$ be a Hilbert subspace.

\begin{dfn}\label{DFN.JFC_Compression_I}
We say that a normal unital $^{*}$-homomorphism $\phi:M\longrightarrow\BII(H)$ is $(N,V)$-compressible if $\phi(N)\subset\BII(V)$ and $\pi_{V}=\phi\lc{}1_{N}\rc\pi_{V}$.
\end{dfn}

\begin{prp}\label{PRP.JFC_Compression_I}
If $\phi:M\longrightarrow\BII(H)$ is $(N,V)$-compressible, then $\phi\vert_{N}:N\longrightarrow\BII(V)$ is a normal unital $^{*}$-homomorphism.
\end{prp}
\begin{proof}
Note $\phi\lc{}1_{N}\rc\in\BII(V)$ and $\pi_{V}=\phi\lc{}1_{N}\rc\pi_{V}$ shows $\phi\lc{}1_{N}\rc{}=\comV\phi\lc{}1_{N}\rc{}=\pi_{V}$. Since $\phi(N)\subset\BII(V)$, we see $\phi\vert_{N}:N\longrightarrow\BII(V)$ is a normal unital $^{*}$-homomorphism.
\end{proof}


\pagebreak


\begin{rem}
Let $\phi:M\longrightarrow\BII(H)$ be $(N,V)$-compressible. For all $x\in L^{0}(N,\tau)_{h}$, we have $\im E_{x,N}\subset N$ as per $1)$ in Definition \ref{DFN.Wstar_CLRA_FC_I} and therefore $\phi\lc\im E_{x,N}\rc\subset\BII(V)$.
\end{rem}

\begin{dfn}\label{DFN.JFC_Compression_II}
Let $\phi:M\longrightarrow\BII(H)$ be $(N,V)$-compressible. For all $x\in L^{0}(N,\tau)_{h}$, we define the push-forward spectral measure $\phi\lc{}E_{x,N}\rc$ of $x$ in $N$ under $\phi$ to $V$ by setting

\begin{align}\label{EQ.DFN.JFC_Compression_II_1}
\phi\lc{}E_{x,N}\rc(Z):=\phi\lc{}E_{x,N}(Z)\rc{}
\end{align}

\noindent for all $Z\in\mathfrak{B}(\mathbb{R})$.
\end{dfn}

\begin{lem}\label{LEM.JFC_Compression_II}
Let $\phi:M\longrightarrow\BII(H)$ be $(N,V)$-compressible. If $x\in L^{0}(N,\tau)_{h}$ and further $T\in\UBII_{V}(H)$ s.t.~$\phi\lc\Gamma_{x,M}\lc{}R_{\pm i}\rc\rc{}=R_{\pm i}(T)$, then we have $\phi\lc\Gamma_{x,N}\lc{}R_{\pm i}\rc\rc{}=R_{\pm i}(\restr{0.925}{T}{V})$ and $\phi\lc{}E_{x,N}\rc{}=E_{\restr{0.925}{T}{V}}$.
\end{lem}
\begin{proof}
For all $y\in N$, we have $\phi\lc{}1_{N}\rc\in\BII(V)$, $\lb\phi\lc{}1_{N}\rc{},\pi_{V}\rb{}=0$ and 

\begin{align}\label{EQ.LEM.JFC_Compression_II_1}
\comV\phi(y)=\comV\lc\phi\lc{}1_{N}\rc\phi(y)\phi\lc{}1_{N}\rc\rc{}=\comV\phi\lc\comunit y\rc{}
\end{align}

\noindent since $\phi$ is $(N,V)$-compressible. If $y=\Gamma_{z,M}(g)$ for $z\in N$ and $g\in L^{\infty}\lc\specM z,dE_{z,M}\rc$, then Equation \ref{EQ.LEM.JFC_Compression_II_1} and $2)$ in Lemma \ref{LEM.Wstar_Compression_Preservation} show

\begin{align}\label{EQ.LEM.JFC_Compression_II_2}
\comV\phi\lc\Gamma_{z,M}(g)\rc{}=\comV\phi\lc\comunit\Gamma_{z,M}(g)\rc{}=\phi\lc\Gamma_{z,N}(g)\rc{}.
\end{align}

Let $x\in L^{0}(N,\tau)_{h}$ and $T\in\UBII_{V}(H)$ s.t.~$\phi\lc\Gamma_{x,M}\lc{}R_{\pm i}\rc\rc{}=R_{\pm i}(T)$. Then using $2)$ in Lemma \ref{LEM.Compression_Preservation_I}, Equation \ref{EQ.LEM.JFC_Compression_II_2} lets us calculate

\begin{align}\label{EQ.LEM.JFC_Compression_II_3}
R_{\pm i}(\restr{0.925}{T}{V})=\comV R_{\pm i}(T)=\comV\phi\lc\Gamma_{x,M}\lc{}R_{\pm i}\rc\rc{}=\phi\lc\Gamma_{x,N}\lc{}R_{\pm i}\rc\rc{}.
\end{align}

\noindent Proposition \ref{PRP.JFC_Compression_I} shows $\phi\circ L_{N}^{-1}:L^{\infty}(N,\tau)\longrightarrow\BII(V)$ is normal unital $^{*}$-homomorphism. Set $\phi_{L,N}:=\phi\circ L_{N}^{-1}$. Using $2)$ in Lemma \ref{LEM.Wstar_CLRA_FC}, Equation \ref{EQ.LEM.JFC_Compression_II_3} implies $\phi_{L,N}\lc{}R_{\pm i}\lc{}L_{x,N}\rc\rc{}=\phi\lc\Gamma_{x,N}\lc{}R_{\pm i}\rc\rc{}=R_{\pm i}(\restr{0.925}{T}{V})$. We see approximating in norm shows

\begin{align}\label{EQ.LEM.JFC_Compression_II_4}
\phi\lc\Gamma_{x,N}(g)\rc{}=\phi_{L,N}\lc{}g\lc{}L_{x,N}\rc\rc{}=g(\restr{0.925}{T}{V})    
\end{align}

\noindent for all $g\in C_{0}(\mathbb{R})$. We have push-forward measure $\phi_{L,N}\lc{}E_{L_{x,N}}\rc$ as per Definition \ref{DFN.FC_Preservation}. Precomposing with $L_{N}^{-1}$ maps Equation \ref{EQ.DFN.JFC_Compression_II_1} to Equation \ref{EQ.DFN.FC_Preservation_1} for $\phi_{L,N}\lc{}E_{L_{x,N}}\rc$, i.e.~

\begin{align}\label{EQ.LEM.JFC_Compression_II_5}
\phi\lc{}E_{x,N}\rc{}=\phi_{L,N}\lc{}E_{L_{x,N}}\rc{}.
\end{align}

\noindent Equation \ref{EQ.LEM.JFC_Compression_II_4} shows Lemma \ref{LEM.FC_Preservation_I} applies to $\phi_{L,N}\lc{}E_{L_{x,N}}\rc$. As such, Equation \ref{EQ.LEM.JFC_Compression_II_5} shows $\phi\lc{}E_{x,N}\rc{}=\phi_{L,N}\lc{}E_{L_{x,N}}\rc{}=E_{\restr{0.925}{T}{V}}$ by Lemma \ref{LEM.FC_Preservation_I}. We therefore know $\phi\lc{}E_{x,N}\rc$ is a spectral measure on $\mathbb{R}$ with values in $\BII(V)$.
\end{proof}

\begin{dfn}\label{DFN.JFC_Compression_III}
Let $\phi:M\longrightarrow\BII(H)$ be $(N,V)$-compressible and $\bpsi:M^{\op}\longrightarrow\BII(H)$ $\lc{}N^{\op},V\rc$-compressible. The pair $(\phi,\bpsi)$ is $(N,V)$-compressible. Set $\phi\otimes_{V}\bpsi:=\phi\vert_{N}\otimes\bpsi\vert_{N^{\op}}$.
\end{dfn}

\begin{thm}\label{THM.JFC_Compression}
Let $N\subset (M,\tau)$ and $V\subset H$ be a Hilbert subspace. Let $(\phi,\bpsi)$ be an $(N,V)$-compressible pair. If $x,y\in L^{0}(N,\tau)_{h}$ and further $T,S\in\UBII_{V}(H)$ commute strongly s.t.~$\phi\lc\Gamma_{x,M}\lc{}R_{\pm i}\rc\rc{}=R_{\pm i}(T)$ and $\bpsi\lc\Gamma_{y,M}\lc{}R_{\pm i}\rc\rc{}=R_{\pm i}(S)$, then we have

\begin{itemize}
\item[1)] $\specN x\times y=\spec \restr{0.925}{T}{V}\times \restr{0.925}{S}{V}$ and $\mathcal{N}\lc{}E_{x,y,N}\rc{}=\mathcal{N}\lc{}E_{\restr{0.925}{T}{V},\restr{0.925}{S}{V}}\rc$,

\item[2)] $\lc\phi\otimes_{V}\bpsi\rc\lc{}g(x,y)\rc{}=g\lc\restr{0.925}{T}{V},\restr{0.925}{S}{V}\rc$ for all $g\in L^{\infty}\lc\specN x\times y,dE_{x,y,N}\rc$,

\item[3)] commutative diagram of normal unital surjective $^{*}$-homomorphisms

\begin{equation}\label{EQ.THM.JFC_Compression_1}
\begin{tikzcd}
L^{\infty}\lc\spec T\times S,dE_{T,S}\rc\arrow[rr,"\Gamma_{T,S}"]\arrow[dddddd,start anchor={[xshift=-1.5cm]},end anchor={[xshift=-1.5cm]},bend right=40,swap,"\res"] & & W^{*}\lc{}T,S\rc\arrow[dddddd,"\comV"] \\
& & \\
L^{\infty}\lc\specM x\times y,dE_{x,y,M}\rc\arrow[r,"\Gamma_{x,y,M}"]\arrow[dd,"\res"]\arrow[uu,swap,"\id"] & W_{M}^{*}(x,y)\arrow[ruu,swap,"\phi\otimes\bpsi"]\arrow[dd,"\comunitunit"] & \\
& & \\
L^{\infty}\lc\specN x\times y,dE_{x,y,N}\rc\arrow[r,"\Gamma_{x,y,N}"]\arrow[dd,"\id"] & W_{N}^{*}(x,y)\arrow[rdd,"\phi\otimes_{V}\bpsi"] & \\
& & \\
L^{\infty}\lc\spec \restr{0.925}{T}{V}\times \restr{0.925}{S}{V},dE_{\restr{0.925}{T}{V},\restr{0.925}{S}{V}}\rc\arrow[rr,"\Gamma_{\restr{0.925}{T}{V},\restr{0.925}{S}{V}}"] & & W^{*}\lc\restr{0.925}{T}{V},\restr{0.925}{S}{V}\rc{} \\
\end{tikzcd}
\end{equation}

\begin{reapply}
\end{reapply}

\noindent with restriction maps given by $\spec \restr{0.925}{T}{V}\times \restr{0.925}{S}{V}\subset\spec T\times S$, $\specN x\times y\subset\specM x\times y$.
\end{itemize}
\end{thm}
\begin{proof}
Let $x,y\in L^{0}(N,\tau)_{h}$ and $T,S\in\UBII_{V}(H)$ commute strongly s.t.~$\phi\lc\Gamma_{x,M}\lc{}R_{\pm i}\rc\rc{}=R_{\pm i}(T)$ and $\bpsi\lc\Gamma_{y,M}\lc{}R_{\pm i}\rc\rc{}=R_{\pm i}(S)$. By construction of $W^{*}$-tensor products, the normal unital $^{*}$-isomorphism $L_{N}\otimes R_{N}:W_{M}^{*}(x,y)\longrightarrow W^{*}\lc{}L_{x,M},R_{y,M}\rc$ has inverse $\lc{}L_{N}\otimes R_{N}\rc^{-1}=L_{N}^{-1}\otimes R_{N}^{-1}$ b . Thus $2)$ in Lemma \ref{LEM.Wstar_CLRA_JFC} shows we have commutative diagram

\smallskip

\begin{equation}\label{EQ.THM.JFC_Compression_2}
\begin{tikzcd}
L^{\infty}\lc\spec L_{x,N}\times R_{y,N},dE_{L_{x,N},R_{y,N}}\rc\arrow[rr,"\Gamma_{L_{x,N},R_{y,N}}"]\arrow[dd,"\id"] & & W^{*}\lc{}L_{x,N},R_{y,N}\rc\arrow[dd,"L_{N}^{-1}\otimes R_{N}^{-1}"] \\
& & \\
L^{\infty}\lc\specN x\times y,dE_{x,y,N}\rc\arrow[rr,"\Gamma_{x,y,N}"] & & W_{N}^{*}(x,y)
\end{tikzcd}
\end{equation}

\medskip

\noindent of normal unital $^{*}$-isomorphisms.\par


\pagebreak


Lemma \ref{LEM.JFC_Compression_II} implies $\phi\lc\Gamma_{x,N}\lc{}R_{\pm i}\rc\rc{}=R_{\pm i}(\restr{0.925}{T}{V})$ and $\bpsi\lc\Gamma_{y,N}\lc{}R_{\pm i}\rc\rc{}=R_{\pm i}\lc\restr{0.925}{S}{V}\rc$. Arguing as in the proof of Lemma \ref{LEM.RP}, mapping $C^{*}$-generators onto and closing in $\sigma$-weak operator topology provides normal unital $^{*}$-isomorphisms $\phi:W_{N}^{*}(x)\longrightarrow W^{*}(\restr{0.925}{T}{V})$ and $\phi:W_{N}^{*}(y)\longrightarrow W^{*}\lc\restr{0.925}{S}{V}\rc$. Corollary \ref{COR.Wstar_TP} lets us tensor these two $^{*}$-isomorphisms to the normal unital $^{*}$-isomorphism $\phi\otimes_{V}\bpsi:W_{N}^{*}(x,y)\longrightarrow W^{*}\lc\restr{0.925}{T}{V},\restr{0.925}{S}{V}\rc$.\par
Arguing as in the proof of Lemma \ref{LEM.JFC_Compression_II}, set $\phi_{L,N}:=\phi\circ L_{N}^{-1}$ and $\bpsi_{R,N}:=\bpsi\circ R_{N}^{-1}$. Lemma \ref{LEM.Wstar_CLRA_II} shows $R_{N}=L_{N}^{\op}$ on $N$. Moreover, $2)$ in Lemma \ref{LEM.Wstar_CLRA_FC} and Lemma \ref{LEM.JFC_Compression_II} show we in fact have normal unital $^{*}$-isomorphisms $\phi_{L,N}:W^{*}\lc{}L_{x,N}\rc\longrightarrow W^{*}(\restr{0.925}{T}{V})$ and $\bpsi_{R,N}:W^{*}\lc{}R_{y,N}\rc\longrightarrow W^{*}\lc\restr{0.925}{S}{V}\rc$ s.t.~$\phi_{L,N}\lc{}E_{L_{x,N}}\rc{}=E_{\restr{0.925}{T}{V}}$ and $\bpsi_{R,N}\lc{}E_{R_{x,N}}\rc{}=E_{\restr{0.925}{S}{V}}$.\par
Thus Lemma \ref{LEM.JFC_Preservation} applies using $\phi_{L,N}$ and $\bpsi_{R,N}$, resp.~their inverses. The concrete analogue of $1)$ hence follows by $1)$ in Lemma \ref{LEM.JFC_Preservation}. We pull back along Diagram \ref{EQ.THM.JFC_Compression_2} to the abstract case. This is $1)$. In our setting, Diagram \ref{EQ.LEM.JFC_Preservation_1} in Lemma \ref{LEM.JFC_Preservation} is the commutative diagram

\smallskip

\begin{equation}\label{EQ.THM.JFC_Compression_3}
\begin{tikzcd}
L^{\infty}\lc\spec L_{x,N}\times R_{y,N},dE_{L_{x,N},R_{y,N}}\rc\arrow[rr,"\Gamma_{L_{x,N},R_{y,N}}"]\arrow[dd,"\phantom{\id}",phantom]\arrow[dddd,"\id"] & & W^{*}\lc{}L_{x,N},R_{y,N}\rc\arrow[dd,"\phantom{L_{N}^{-1}\otimes R_{N}^{-1}}",phantom]\arrow[dddd,"\phi_{L,N}\otimes\bpsi_{R,N}"] \\
& & \\
\phantom{L^{\infty}\lc\specN x\times y,dE_{x,y,N}\rc{}}\arrow[rr,"\phantom{\Gamma_{x,y,N}}",phantom]\arrow[dd,"\phantom{\id}",phantom] & & \phantom{W_{N}^{*}(x,y)}\arrow[dd,"\phantom{\phi\otimes_{V}\bpsi}",phantom] \\
& & \\
L^{\infty}\lc\spec \restr{0.925}{T}{V}\times \restr{0.925}{S}{V},dE_{\restr{0.925}{T}{V},\restr{0.925}{S}{V}}\rc\arrow[rr,"\Gamma_{\restr{0.925}{T}{V},\restr{0.925}{S}{V}}"] & & W^{*}\lc\restr{0.925}{T}{V},\restr{0.925}{S}{V}\rc{}
\end{tikzcd}
\end{equation}

\medskip

\noindent of normal unital $^{*}$-isomorphisms. Using $\sigma$-weak closure, we directly verify

\begin{align}\label{EQ.THM.JFC_Compression_4}
\phi_{L,N}\otimes\bpsi_{R,N}=\lc\phi\otimes_{V}\bpsi\rc\circ\lc{}L_{N}^{-1}\otimes R_{N}^{-1}\rc{}    
\end{align}

\noindent on elementary tensors. Equation \ref{EQ.THM.JFC_Compression_4} shows Diagram \ref{EQ.THM.JFC_Compression_3} factors into the upper and lower diagrams

\smallskip

\begin{equation}\label{EQ.THM.JFC_Compression_5}
\begin{tikzcd}
L^{\infty}\lc\spec L_{x,N}\times R_{y,N},dE_{L_{x,N},R_{y,N}}\rc\arrow[rr,"\Gamma_{L_{x,N},R_{y,N}}"]\arrow[dd,"\id"]\arrow[dddd,"\phantom{\id}",phantom] & & W^{*}\lc{}L_{x,N},R_{y,N}\rc\arrow[dd,"L_{N}^{-1}\otimes R_{N}^{-1}"]\arrow[dddd,"\phantom{\phi_{L,N}\otimes\bpsi_{R,N}}",phantom]\\
& & \\
L^{\infty}\lc\specN x\times y,dE_{x,y,N}\rc\arrow[rr,"\Gamma_{x,y,N}"]\arrow[dd,"\id"] & & W_{N}^{*}(x,y)\arrow[dd,"\phi\otimes_{V}\bpsi"] \\
& & \\
L^{\infty}\lc\spec \restr{0.925}{T}{V}\times \restr{0.925}{S}{V},dE_{\restr{0.925}{T}{V},\restr{0.925}{S}{V}}\rc\arrow[rr,"\Gamma_{\restr{0.925}{T}{V},\restr{0.925}{S}{V}}"] & & W^{*}\lc\restr{0.925}{T}{V},\restr{0.925}{S}{V}\rc{}
\end{tikzcd}
\end{equation}

\medskip

\noindent of normal unital $^{*}$-isomorphisms. The outer diagram in Diagram \ref{EQ.THM.JFC_Compression_5} is therefore given by Diagram \ref{EQ.THM.JFC_Compression_3}, whereas the upper one is Diagram \ref{EQ.THM.JFC_Compression_2}. Thus both outer and upper diagrams commute, hence the lower diagram in Diagram \ref{EQ.THM.JFC_Compression_5} commutes.\par
We apply the above discussion to $(N,V)$ and its special case $(N,V)=\lc{}M,H\rc$. We thus combine both to have commutative diagram

\vspace{-0.075cm}
\begin{equation}\label{EQ.THM.JFC_Compression_6}
\begin{tikzcd}
L^{\infty}\lc\spec T\times S,dE_{T,S}\rc\arrow[rr,"\Gamma_{T,S}"]\arrow[dddddd,start anchor={[xshift=-1.5cm]},end anchor={[xshift=-1.5cm]},bend right=40,swap,"\phantom{\res}",phantom] & & W^{*}\lc{}T,S\rc\arrow[dddddd,"\phantom{\comV}",phantom] \\
& & \\
L^{\infty}\lc\specM x\times y,dE_{x,y,M}\rc\arrow[r,"\Gamma_{x,y,M}"]\arrow[dd,"\phantom{\res}",phantom]\arrow[uu,swap,"\id"] & W_{M}^{*}(x,y)\arrow[ruu,swap,"\phi\otimes\bpsi"]\arrow[dd,"\phantom{\comunitunit}",phantom] & \\
& & \\
L^{\infty}\lc\specN x\times y,dE_{x,y,N}\rc\arrow[r,"\Gamma_{x,y,N}"]\arrow[dd,"\id"] & W_{N}^{*}(x,y)\arrow[rdd,"\phi\otimes_{V}\bpsi"] & \\
& & \\
L^{\infty}\lc\spec \restr{0.925}{T}{V}\times \restr{0.925}{S}{V},dE_{\restr{0.925}{T}{V},\restr{0.925}{S}{V}}\rc\arrow[rr,"\Gamma_{\restr{0.925}{T}{V},\restr{0.925}{S}{V}}"] & & W^{*}\lc\restr{0.925}{T}{V},\restr{0.925}{S}{V}\rc{} \\
\end{tikzcd}
\end{equation}
\vspace{-0.15cm}

\noindent of normal unital $^{*}$-isomorphisms. We further have commutative diagram

\vspace{-0.075cm}
\begin{equation}\label{EQ.THM.JFC_Compression_7}
\begin{tikzcd}
L^{\infty}\lc\spec T\times S,dE_{T,S}\rc\arrow[rr,"\Gamma_{T,S}"]\arrow[dddddd,start anchor={[xshift=-1.5cm]},end anchor={[xshift=-1.5cm]},bend right=40,swap,"\res"] & & W^{*}\lc{}T,S\rc\arrow[dddddd,"\comV"] \\
& & \\
L^{\infty}\lc\specM x\times y,dE_{x,y,M}\rc\arrow[r,"\Gamma_{x,y,M}"]\arrow[dd,"\res"]\arrow[uu,swap,"\phantom{\id}",phantom] & W_{M}^{*}(x,y)\arrow[ruu,swap,"\phantom{\phi\otimes\bpsi}",phantom]\arrow[dd,"\comunitunit"] & \\
& & \\
L^{\infty}\lc\specN x\times y,dE_{x,y,N}\rc\arrow[r,"\Gamma_{x,y,N}"]\arrow[dd,"\phantom{\id}",phantom] & W_{N}^{*}(x,y)\arrow[rdd,"\phantom{\phi\otimes_{V}\bpsi}",phantom] & \\
& & \\
L^{\infty}\lc\spec \restr{0.925}{T}{V}\times \restr{0.925}{S}{V},dE_{\restr{0.925}{T}{V},\restr{0.925}{S}{V}}\rc\arrow[rr,"\Gamma_{\restr{0.925}{T}{V},\restr{0.925}{S}{V}}"] & & W^{*}\lc\restr{0.925}{T}{V},\restr{0.925}{S}{V}\rc{} \\
\end{tikzcd}
\end{equation}
\vspace{-0.15cm}

\noindent of normal unital $^{*}$-homomorphisms. Indeed, note Diagram \ref{EQ.LEM.Compression_Preservation_II_2} as per Lemma \ref{LEM.Compression_Preservation_II} is the outer diagram in Diagram \ref{EQ.THM.JFC_Compression_7}, whereas the left diagram in Diagram \ref{EQ.COR.Wstar_Compression_Preservation_III_1} as per Corollary \ref{COR.Wstar_Compression_Preservation_III} is the inner one.\par


\pagebreak


We combine Diagram \ref{EQ.THM.JFC_Compression_6} and Diagram \ref{EQ.THM.JFC_Compression_7} to Diagram \ref{EQ.THM.JFC_Compression_1}. Commutativity of the latter therefore implies both $2)$ and $3)$ follow if the diagram

\begin{equation}\label{EQ.THM.JFC_Compression_8}
\begin{tikzcd}
\phantom{L^{\infty}\lc\spec T\times S,dE_{T,S}\rc{}}\arrow[rr,"\phantom{\Gamma_{T,S}}",phantom]\arrow[dddddd,start anchor={[xshift=-1.5cm]},end anchor={[xshift=-1.5cm]},bend right=40,swap,"\phantom{\res}",phantom] & & W^{*}\lc{}T,S\rc\arrow[dddddd,"\comV"] \\
& & \\
\phantom{L^{\infty}\lc\specM x\times y,dE_{x,y,M}\rc{}}\arrow[r,"\phantom{\Gamma_{x,y,M}}",phantom]\arrow[dd,"\phantom{\res}",phantom]\arrow[uu,swap,"\phantom{\id}",phantom] & W_{M}^{*}(x,y)\arrow[ruu,swap,"\phi\otimes\bpsi"]\arrow[dd,"\comunitunit"] & \\
& & \\
\phantom{L^{\infty}\lc\specN x\times y,dE_{x,y,N}\rc{}}\arrow[r,"\phantom{\Gamma_{x,y,N}}",phantom]\arrow[dd,"\phantom{\id}",phantom] & W_{N}^{*}(x,y)\arrow[rdd,"\phi\otimes_{V}\bpsi"] & \\
& & \\
\phantom{L^{\infty}\lc\spec \restr{0.925}{T}{V}\times \restr{0.925}{S}{V},dE_{\restr{0.925}{T}{V},\restr{0.925}{S}{V}}\rc{}}\arrow[rr,"\phantom{\Gamma_{\restr{0.925}{T}{V},\restr{0.925}{S}{V}}}",phantom] & & W^{*}\lc\restr{0.925}{T}{V},\restr{0.925}{S}{V}\rc{} \\
\end{tikzcd}
\end{equation}

\noindent of normal unital $^{*}$-homomorphisms commutes. Normality reduces to commutativity on elementary tensors. Note $\comunitunit=\comunit\otimes\comunit$. Equation \ref{EQ.LEM.JFC_Compression_II_2} for $\phi_{L,N}$ and $\bpsi_{R,N}$ implies commutativity on elementary tensors.
\end{proof}

\begin{cor}\label{COR.JFC_Compression_I}
Assume the setting of Theorem \ref{THM.JFC_Compression}. For all real $g\in\SII\lc{}E_{x,y,N}\rc$ s.t~

\begin{itemize}
\item[1)] $(t,s)\mapsto g_{\varepsilon}(t,s):=g\lc{}t+\varepsilon,\s+\varepsilon\rc$ lies in $C_{b}\lc\spec T_{0}\times S_{0}\rc$ for all $\varepsilon>0$,

\item[2)] $g\lc\restr{0.925}{T}{V},\restr{0.925}{S}{V}\rc{}=\sr$-$\lim_{\varepsilon\downarrow 0}g_{\varepsilon}\lc\restr{0.925}{T}{V},\restr{0.925}{S}{V}\rc$ on $V$,
\end{itemize}

\noindent we have $g\in\SII\lc{}E_{\restr{0.925}{T}{V},\restr{0.925}{S}{V}}\rc$ with $g\lc\restr{0.925}{T}{V},\restr{0.925}{S}{V}\rc{}=\sr$-$\lim_{\varepsilon\downarrow 0}\hspace{0.025cm}\lc\phi\otimes_{V}\bpsi\rc\lc{}g_{\varepsilon}(x,y)\rc$ on $V$.
\end{cor}
\begin{proof}
Apply Lemma \ref{LEM.JFC_Compression_II} and Theorem \ref{THM.JFC_Compression} to reduce to Corollary \ref{COR.JFC_Preservation}.
\end{proof}

\begin{dfn}\label{DFN.JFC_Compression_IV}
Assume the setting of Theorem \ref{THM.JFC_Compression}.

\begin{itemize}
\item[1)] We call $\Gamma_{x,y,N}^{\phi,\bpsi}:=\lc\phi\otimes_{V}\bpsi\rc\circ\Gamma_{x,y,N}$ the bounded measurable joint functional calculus of $x\otimes y$ in $N\otimes N^{\op}$ under $\phi\otimes_{V}\bpsi$.

\item[2)] Let $\mathcal{S}_{V}\lc{}E_{x,y,N}\rc$ be the set of all real $g\in\SII\lc{}E_{x,y,N}\rc$ s.t.~$1)$ and $2)$ in Corollary \ref{COR.JFC_Compression_I} are satisfied. For all $g\in\mathcal{S}_{V}\lc{}E_{x,y,N}\rc$, set

\begin{align}\label{EQ.DFN.JFC_Compression_IV_1}
\Gamma_{x,y,N}^{\phi,\bpsi}(g):=g\lc\restr{0.925}{T}{V},\restr{0.925}{S}{V}\rc{}.
\end{align}

\begin{reapply}
\end{reapply}

\item[3)] We call $\Gamma_{x,y,N}^{\phi,\bpsi}:\mathcal{S}_{V}\lc{}E_{x,y,N}\rc\longrightarrow\UBII\lc{}V\rc_{h}$ the joint functional calculus of $x\otimes y$ in $N\otimes N^{\op}$ under $\phi\otimes_{V}\bpsi$.
\end{itemize}
\end{dfn}


\pagebreak


\begin{cor}\label{COR.JFC_Compression_II}
Let $N_{0}\subset N_{1}\subset (M,\tau)$ and $V_{0}\subset V_{1}\subset H$ be Hilbert subspaces. Let $(\phi,\bpsi)$ be an $\lc{}N_{0},V_{0}\rc$-~and $\lc{}N_{1},V_{1}\rc$-compressible pair. If $x,y\in L^{0}\lc{}N_{0},\tau\rc_{h}$ and $T,S\in\UBII_{V}(H)$ commute strongly s.t.~$\phi\lc\Gamma_{x,M}\lc{}R_{\pm i}\rc\rc{}=R_{\pm i}(T)$ and $\bpsi\lc\Gamma_{y,M}\lc{}R_{\pm i}\rc\rc{}=R_{\pm i}(S)$, then we have

\begin{itemize}
\item[1)] $\specNone x\times y\subset\specM x\times y$ and $\mathcal{N}\lc{}E_{x,y,M}\rc\subset\mathcal{N}\lc{}E_{x,y,N_{1}}\rc$,

\item[2)] $\specNnull x\times y\subset\specNone x\times y$ and $\mathcal{N}\lc{}E_{x,y,N_{1}}\rc\subset\mathcal{N}\lc{}E_{x,y,N_{0}}\rc$,

\item[3)] commutative diagram of normal unital surjective $^{*}$-homomorphisms

\begin{equation}\label{EQ.COR.JFC_Compression_II_1}
\begin{tikzcd}
L^{\infty}\lc\specM x\times y,dE_{x,y,M}\rc\arrow[rr, "\Gamma_{x,y}^{\phi,\bpsi}"]\arrow[dd,"\res"] & & \BII(H)\arrow[dd,"\comVone"] \\
& & \\
L^{\infty}\lc\specNone x\times y,dE_{x,y,N_{1}}\rc\arrow[rr, "\Gamma_{x,y,N_{1}}^{\phi,\bpsi}"]\arrow[dd,"\res"] & & \BII\lc{}V_{1}\rc\arrow[dd,"\comVnull"] \\
& & \\
L^{\infty}\lc\specNnull x\times y,dE_{x,y,N_{0}}\rc\arrow[rr, "\Gamma_{x,y,N_{0}}^{\phi,\bpsi}"] & & \BII\lc{}V_{0}\rc{}
\end{tikzcd}
\end{equation}

\begin{reapply}
\end{reapply}

\noindent with restriction maps given by $\specNone x\times y\subset\specM x\times y$, $\specNnull x\times y\subset\specNone x\times y$.
\end{itemize}
\end{cor}
\begin{proof}
Get $1)$ and $2)$ by Corollary \ref{COR.Wstar_Compression_Preservation_III}. Lemma \ref{LEM.JFC_Compression_II} shows $\phi\lc\Gamma_{x,N_{1}}\lc{}R_{\pm i}\rc\rc{}=R_{\pm i}\lc\restr{0.925}{T}{V_{1}}\rc$ and $\bpsi\lc\Gamma_{y,N_{1}}\lc{}R_{\pm i}\rc\rc{}=R_{\pm i}\lc\restr{0.925}{S}{V_{1}}\rc$. Theorem \ref{THM.JFC_Compression} applies using $N_{1}\subset (M,\tau)$ and $V_{1}$, as well as $N_{0}\subset \lc{}N_{1},\tau\rc$ and $V_{0}$. We use $\phi$ and $\bpsi$ in both cases. Using $3)$ in Corollary \ref{COR.Wstar_Compression_Preservation_III} and $3)$ in Theorem \ref{THM.JFC_Compression} applied twice accordingly, we directly verify Diagram \ref{EQ.COR.JFC_Compression_II_1}.
\end{proof}

\begin{cor}\label{COR.JFC_Compression_III}
Assume the setting of Theorem \ref{THM.JFC_Compression}. If $x,y\in L^{0}(N,\tau)_{+}$, $\alpha,\beta\geq 0$ and $g\in C_{b}\lc{}[0,\infty)\times [0,\infty)\rc$, then

\begin{align}\label{EQ.COR.JFC_Compression_III_1}
\comunitunit\lc\Gamma_{x+\alpha 1_{N}^{\perp},y+\beta 1_{N}^{\perp},M}(g)\rc{}=\Gamma_{x,y,N}(g).
\end{align}
\end{cor}
\begin{proof}
Let $x,y\in L^{0}(N,\tau)_{+}$, $\alpha,\beta\geq 0$ and $g\in C_{b}\lc{}[0,\infty)\times [0,\infty)\rc$. Proposition \ref{PRP.Compression_Abstract_Bd} shows abstract compression maps are normal. By normality and Lemma \ref{LEM.Wstar_CLRA_JFC_SR}, we assume $\lset{}0\rset\in\mathcal{N}\lc{}E_{x,N}\rc\cap\mathcal{N}\lc{}E_{y,N}\rc$ without loss of generality.\par
Set $Z_{x}:=\specM x\setminus\lset{}0\rset$ and $Z_{y}:=\specM y\setminus\lset{}0\rset$. Note $1)$ in Corollary \ref{COR.Wstar_Compression_Preservation_I} shows

\begin{align}\label{EQ.COR.JFC_Compression_III_2}
\Gamma_{x,N}\big(\chi_{Z_{x}}\big)=1_{N},\ \Gamma_{y,N}\big(\chi_{Z_{y}}\big)=1_{N}.
\end{align}

\noindent Using Corollary \ref{COR.Wstar_Compression_Preservation_II}, Equation \ref{EQ.COR.JFC_Compression_III_2} implies

\begin{align}\label{EQ.COR.JFC_Compression_III_3}
\Gamma_{x,M}\big(\chi_{Z_{x}}\big)=\Gamma_{x,N}\big(\chi_{Z_{x}}\big)=1_{N}=\Gamma_{y,N}\big(\chi_{Z_{y}}\big)=\Gamma_{y,M}\big(\chi_{Z_{y}}\big).
\end{align}

Equation \ref{EQ.COR.JFC_Compression_III_3} yields $1_{N}^{\perp}\in W_{M}^{*}(x)\cap W_{M}^{*}(y)$ and

\begin{align}\label{EQ.COR.JFC_Compression_III_4}
\Gamma_{x,M}(\delta_{0})=1_{N}^{\perp}=\Gamma_{y,M}(\delta_{0}).
\end{align}

\noindent For all $t,s\in\mathbb{R}$, let $g^{\alpha,\beta}(t,s):=g\lc{}t+\alpha\delta_{0}(t),s+\beta\delta_{0}(s)\rc$. We obtain $g^{\alpha,\beta}\in C_{b}\lc\mathbb{R}\times\mathbb{R}\rc$ and

\begin{align}\label{EQ.COR.JFC_Compression_III_5}
\Gamma_{x,y,N}\lc{}g^{\alpha,\beta}\rc{}=\Gamma_{x,y,N}(g).    
\end{align}

\noindent Using Theorem \ref{THM.JFC_Compression}, we calculate

\begin{align}\label{EQ.COR.JFC_Compression_III_6}
\comunitunit\lc\Gamma_{x+\alpha 1_{N}^{\perp},y+\beta 1_{N}^{\perp},M}(g)\rc{}=\comunitunit\lc\Gamma_{x,y,M}\lc{}g^{\alpha,\beta}\rc\rc{}= \Gamma_{x,y,N}\lc{}g^{\alpha,\beta}\rc{}.
\end{align}

\noindent Equation \ref{EQ.COR.JFC_Compression_III_5} and Equation \ref{EQ.COR.JFC_Compression_III_6} show Equation \ref{EQ.COR.JFC_Compression_III_1}.
\end{proof}


\chapter{Clifford Calculations}\label{APP.C}

We give calculations in Clifford algebras for our discussion. In Section \ref{SEC.C_Identities}, Lemma \ref{LEM.Wstar_Derivation_QG_Intertwining_Clifford_Identities} gives three identities for twisted dynamic quantum gradients induced by intertwining sets of Clifford generators. In Section \ref{SEC.C_Implementation_Fock}, Lemma \ref{LEM.QOT_Second_Quantisation_Implementation_Fock} gives, in detail, implementation of Bogoliubov automorphisms as per Equation \ref{EQ.BSP.QOT_Type_II_5} on anti-symmetric Fock space.


\section{Identities for intertwining sets of Clifford generators}\label{SEC.C_Identities}

We use the three identities in Lemma \ref{LEM.Wstar_Derivation_QG_Intertwining_Clifford_Identities} to prove Lemma \ref{LEM.Wstar_Derivation_QG_Intertwining_Clifford}.

\begin{lem}\label{LEM.Wstar_Derivation_QG_Intertwining_Clifford_Identities}
Assume the setting of Lemma \ref{LEM.Wstar_Derivation_QG_Intertwining_Clifford}.

\begin{itemize}
\item[1)] For all $n,k\in\lset{}1,\ldots,m\rset$, we have

\begin{align}\label{EQ.LEM.Wstar_Derivation_QG_Intertwining_Clifford_Identities_1}
\partial_{n}\partial_{k}=-\partial_{k}\partial_{n}.
\end{align}

\begin{reapply}
\end{reapply}

\item[2)] For all $n,k\in\lset{}1,\ldots,m\rset$ s.t.~$n\neq k$, we have

\begin{align}\label{EQ.LEM.Wstar_Derivation_QG_Intertwining_Clifford_Identities_2}
\partial_{n}\mathrlap{\phantom{\partial}^{*}}\partial_{k}=-\mathrlap{\phantom{\partial}^{*}}\partial_{k}\partial_{n}.
\end{align}

\begin{reapply}
\end{reapply}

\item[3)] For all $n\in\lset{}1,\ldots,m\rset$, we have

\begin{align}\label{EQ.LEM.Wstar_Derivation_QG_Intertwining_Clifford_Identities_3}
\partial_{n}\Delta_{n}=4C\partial_{n}.
\end{align}

\begin{reapply}
\end{reapply}

\end{itemize}
\end{lem}
\begin{proof}
For all $n\in\lset{}1,\ldots,m\rset$, Example \ref{BSP.Wstar_Derivation_QG_Intertwining} shows $L_{d_{n}}\in\BII\lc{}L^{2}(A,\tau)\rc_{h}$ is $\phi$-intertwining with $\sgn\lc{}L_{d_{n}}\rc{}=-1$. We pull back along $L^{-1}$. Note $\phi$ is involutive by hypothesis. Let $n,k\in\lset{}1,\ldots,m\rset$ and $x\in A_{0}$. Using $\lset{}d_{n}\rset_{n=1}^{m}\in L^{\infty}(A,\tau)_{h}$, Corollary \ref{COR.Wstar_Derivation_QG_Intertwining_I} shows 

\begin{align}\label{EQ.LEM.Wstar_Derivation_QG_Intertwining_Clifford_Identities_4}
\partial_{n}x=d_{n}x-\phi(x)d_{n},\ \mathrlap{\phantom{\partial}^{*}}\partial_{n}x=d_{n}x+\phi(x)d_{n}
\end{align}

\noindent and

\begin{align}\label{EQ.LEM.Wstar_Derivation_QG_Intertwining_Clifford_Identities_5}
\Delta_{n}=d_{n}^{2}x+xd_{n}^{2}-2d_{n}\phi(x)d_{n}.
\end{align}


\pagebreak


Note $1.2)$ in Definition \ref{DFN.Wstar_Derivation_QG_Intertwining_Clifford} gives the Clifford relation $d_{n}d_{k}+d_{k}d_{n}=2C\delta_{nk}1_{A}$. We use the first identity in Equation \ref{EQ.LEM.Wstar_Derivation_QG_Intertwining_Clifford_Identities_4} to calculate

\begin{align}\label{EQ.LEM.Wstar_Derivation_QG_Intertwining_Clifford_Identities_6}
\partial_{n}\partial_{k}x=d_{n}d_{k}x-d_{n}\phi(x)d_{k}+d_{k}\phi(x)d_{n}-xd_{k}d_{n}.
\end{align}

\noindent Interchanging $n$ and $k$ in Equation \ref{EQ.LEM.Wstar_Derivation_QG_Intertwining_Clifford_Identities_6} yields $\partial_{k}\partial_{n}x$. We apply the Clifford relation to the first and fourth summand on the right-hand side of Equation \ref{EQ.LEM.Wstar_Derivation_QG_Intertwining_Clifford_Identities_6}. We obtain $-\partial_{k}\partial_{n}x$. Equation \ref{EQ.LEM.Wstar_Derivation_QG_Intertwining_Clifford_Identities_6} therefore implies Equation \ref{EQ.LEM.Wstar_Derivation_QG_Intertwining_Clifford_Identities_1} by Clifford relations. Get $1)$. We likewise use both identities in Equation \ref{EQ.LEM.Wstar_Derivation_QG_Intertwining_Clifford_Identities_4} to calculate

\begin{align}\label{EQ.LEM.Wstar_Derivation_QG_Intertwining_Clifford_Identities_7}
\partial_{n}\mathrlap{\phantom{\partial}^{*}}\partial_{k}x=d_{n}d_{k}x+d_{n}\phi(x)d_{k}+d_{k}\phi(x)d_{n}+xd_{k}d_{n}
\end{align}

\noindent and

\begin{align}\label{EQ.LEM.Wstar_Derivation_QG_Intertwining_Clifford_Identities_8}
\mathrlap{\phantom{\partial}^{*}}\partial_{k}\partial_{n}x=d_{k}d_{n}x-d_{k}\phi(x)d_{n}-d_{n}\phi(x)d_{k}+xd_{n}d_{k}.
\end{align}

\noindent If $n\neq k$, then $d_{n}d_{k}=-d_{k}d_{n}$. Using the latter, we see Equation \ref{EQ.LEM.Wstar_Derivation_QG_Intertwining_Clifford_Identities_7} and Equation \ref{EQ.LEM.Wstar_Derivation_QG_Intertwining_Clifford_Identities_8} imply Equation \ref{EQ.LEM.Wstar_Derivation_QG_Intertwining_Clifford_Identities_2} by Clifford relations. Get $2)$.\par
If $n=k$, then $d_{n}^{2}=C1_{A}$. Equation \ref{EQ.LEM.Wstar_Derivation_QG_Intertwining_Clifford_Identities_5} is $\Delta_{n}x=2Cx-2d_{n}\phi(x)d_{n}$ in this case. We use the latter and the first identity in Equation \ref{EQ.LEM.Wstar_Derivation_QG_Intertwining_Clifford_Identities_4} to calculate

\begin{align}\label{EQ.LEM.Wstar_Derivation_QG_Intertwining_Clifford_Identities_9}
\partial_{n}\Delta_{n}x=2Cd_{n}x-2C\phi(x)d_{n}-2C\phi(x)d_{n}+2Cd_{n}x=4C\partial_{n}x.
\end{align}

\noindent Equation \ref{EQ.LEM.Wstar_Derivation_QG_Intertwining_Clifford_Identities_9} shows Equation \ref{EQ.LEM.Wstar_Derivation_QG_Intertwining_Clifford_Identities_3}. Get $3)$.
\end{proof}


\section{Implementation on anti-symmetric Fock space}\label{SEC.C_Implementation_Fock}

We use Lemma \ref{LEM.QOT_Second_Quantisation_Implementation_Fock} to derive explicit formula for Equation \ref{EQ.BSP.QOT_Type_II_5} in Example \ref{BSP.QOT_Second_Quantisation}.

\begin{lem}\label{LEM.QOT_Second_Quantisation_Implementation_Fock}
Assume the setting of Example \ref{BSP.QOT_Second_Quantisation}. For all $t\in\mathbb{R}$ and $x\in\mathcal{A}(H)$, get

\begin{align}\label{EQ.LEM.QOT_Second_Quantisation_Implementation_Fock_1}
\rho_{J}\lc\Cliff\lc{}e^{itD}\rc{}(x)\rc{}=\bigwedge e^{it\absv{1.15}{D}}\rho_{J}(x)\bigwedge e^{-it\absv{1.15}{D}}\in\rho_{J}\lc\AII(H)\rc{}.
\end{align} 
\end{lem}
\begin{proof}
We solve the associated implementation problem at each time, i.e.~we show each $\Cliff\lc{}e^{itD}\rc$ to be the Bogoliubov automorphism implemented on $\FII(H[J])$ using $\bigwedge e^{it\absv{1.15}{D}}$ as its unique unitary operator \lc{}cf.~Section 3.2 and Section 3.3 in \cite{BK.Ply_Rob.1994.Clifford_Algebras}\rc{}. For this, we construct suitable unitary equivalences of faithful unital $^{*}$-representations of $\AII(H)$ over $\FII(H[J])$ using cyclic vectors.\par


\pagebreak


For all $t\in\mathbb{R}$, set $\rho_{J}^{t}:=\rho_{J}\circ\Cliff\lc{}e^{itD}\rc$ and note a cyclic unit vector $\Omega\in H$ of $\rho_{J}^{t}$ satisfies the $J$-vacuum condition if 

\begin{align}\label{EQ.LEM.QOT_Second_Quantisation_Implementation_Fock_2}
\rho_{J}^{t}\big(u+iJ(u)\big)(\Omega)=0
\end{align}

\noindent for all $u\in H$. Let $\Omega_{J}$ be the Fock state of $\mathcal{F}(H[J])$. It is the unique cyclic unit vector of $\rho_{J}^{0}:=\rho_{J}$ satisfying the $J$-vacuum condition. For all $t\in\mathbb{R}$, set $\Omega_{J}^{t}:=\bigwedge e^{it\absv{1.15}{D}}(\Omega_{J})$.\par
Let $t\in\mathbb{R}$. Theorem 3.2.5 in \cite{BK.Ply_Rob.1994.Clifford_Algebras} states $e^{it\absv{1.15}{D}}\in\UII\lc\BII(H)\rc$ is implemented on $\mathcal{F}(H[J])$ by $\bigwedge e^{it\absv{1.15}{D}}\in\UII\lc\BII\lc\mathcal{F}(H[J])\rc\rc$. Using $H\cong b(H)\subset\AII(H)$ as set of generators and by norm continuity, Equation \ref{EQ.LEM.QOT_Second_Quantisation_Implementation_Fock_1} reduces to $\rho_{J}^{t}=\rho_{J}\circ e^{itD}$ and

\begin{align}\label{EQ.LEM.QOT_Second_Quantisation_Implementation_Fock_3}
\rho_{J}^{t}(u)=\bigwedge e^{it\absv{1.15}{D}}\rho_{J}(u)\bigwedge e^{-it\absv{1.15}{D}}\in\rho_{J}\lc\AII(H)\rc{}
\end{align}

\noindent for all $u\in H$. Note $\lb{}e^{itD},J\rb{}=0$ ensures Theorem 3.3.3 in \cite{BK.Ply_Rob.1994.Clifford_Algebras} yields implementation as per Equation \ref{EQ.LEM.QOT_Second_Quantisation_Implementation_Fock_3} for unique but unspecified unitary operators. We require $\bigwedge e^{it\absv{1.15}{D}}$ to be the unique unitary operator used. Since $\Omega$ is a cyclic unit vector of $\rho_{J}$, we know $\Omega_{J}^{t}$ is a cyclic unit vector of $\rho_{J}^{t}$ s.t.~$\bigwedge e^{it\absv{1.15}{D}}(\Omega_{J})=\Omega_{J}^{t}$ by construction. If $\Omega_{J}^{t}$ satisfies the $J$-vacuum condition, then Theorem 2.4.7 in \cite{BK.Ply_Rob.1994.Clifford_Algebras} shows Equation \ref{EQ.LEM.QOT_Second_Quantisation_Implementation_Fock_3}.\par
We show the $J$-vacuum condition for $\Omega_{J}^{t}$. Since $\lb{}e^{itD},J\rb{}=0$, we calculate

\begin{align}\label{EQ.LEM.QOT_Second_Quantisation_Implementation_Fock_4}
\rho_{J}\lc{}e^{itD}u\rc(\Omega_{J})=-i\rho_{J}\lc{}J\lc{}e^{itD}u\rc\rc(\Omega_{J})=\rho_{J}\lc{}e^{itD}\lc{}P_{+}-P_{-}\rc{}(u)\rc(\Omega_{J})
\end{align}

\noindent for all $u\in H$. We have $e^{itD}\lc{}P_{+}-P_{-}\rc{}=e^{it\absv{1.15}{D}}$ on non-negative, and $e^{itD}\lc{}P_{+}-P_{-}\rc{}=e^{it\absv{1.15}{D}}-2I$ on negative eigenvalues. For all $v\in H$, note $P_{-}(v)=v$ implies $J(v)=-iv$ and therefore $2v=v+iJ(v)$. Using the $J$-vacuum condition for $\Omega_{J}$, Equation \ref{EQ.LEM.QOT_Second_Quantisation_Implementation_Fock_4} implies

\begin{align}\label{EQ.LEM.QOT_Second_Quantisation_Implementation_Fock_5}
\rho_{J}\lc{}e^{itD}u\rc(\Omega_{J})=\rho_{J}\big(e^{it\absv{1.15}{D}}u\big)(\Omega_{J})   
\end{align}

\noindent for all $u\in H$. Using implementation of $e^{it\absv{1.15}{D}}$ on $\FII(H[J])$ by $\bigwedge e^{it\absv{1.15}{D}}$, Equation \ref{EQ.LEM.QOT_Second_Quantisation_Implementation_Fock_5} lets us calculate

\begin{align}\label{EQ.LEM.QOT_Second_Quantisation_Implementation_Fock_6}
\rho_{J}\lc{}e^{itD}u\rc(\Omega_{J})=\rho_{J}\big(e^{it\absv{1.15}{D}}u\big)(\Omega_{J})=\lc\bigwedge e^{it\absv{1.15}{D}}\rho_{J}(u)\bigwedge e^{-it\absv{1.15}{D}}\rc{}(\Omega_{J})
\end{align}

\noindent for all $u\in H$. Equation \ref{EQ.LEM.QOT_Second_Quantisation_Implementation_Fock_6} shows

\begin{align}\label{EQ.LEM.QOT_Second_Quantisation_Implementation_Fock_7}
\bigwedge e^{it\absv{1.15}{D}}\lc\rho_{J}(u)(\Omega_{J})\rc{}=\bigwedge e^{it\absv{1.15}{D}}\lc\rho_{J}\lc{}e^{-itD}e^{itD}u\rc(\Omega_{J})\rc{}=\rho_{J}^{t}(u)(\Omega_{J}^{t})
\end{align}

\noindent for all $u\in H$. Using $u:=v+iJ(v)$ for all $v\in H$, the $J$-vacuum condition for $\Omega_{J}$ and Equation \ref{EQ.LEM.QOT_Second_Quantisation_Implementation_Fock_7} imply $\Omega_{J}^{t}$ satisfies the $J$-vacuum condition. Thus Theorem 2.4.7 in \cite{BK.Ply_Rob.1994.Clifford_Algebras} shows Equation \ref{EQ.LEM.QOT_Second_Quantisation_Implementation_Fock_3}, hence Equation \ref{EQ.LEM.QOT_Second_Quantisation_Implementation_Fock_1}.
\end{proof}
\end{appendices}


\markboth{BIBLIOGRAPHY}{}
\bibliography{references}


\end{document}